\documentclass[12pt,french,canon,utf8,nothms]{aclbook}

\makeindex

 \usepackage{macros}

\usepackage{tikz,tikz-cd}
\tikzset{commutative diagrams/arrow style=math font}
\DeclareUnicodeCharacter{00A0}{\relax}

\title[Formes et courants sur les espaces de Berkovich]
      {Formes différentielles réelles et courants
	sur les espaces de Berkovich}

\author{Antoine Chambert-Loir}
\address{%
Université Paris Cité, Institut de Mathématiques de Jussieu-Paris Rive Gauche, F-75013, Paris, France}
\email{antoine.chambert-loir@math.univ-paris-diderot.fr}

\author{Antoine Ducros}
\address{%
Sorbonne Université,  Université de Paris, CNRS, Institut de Mathématiques de Jussieu-Paris Rive Gauche, IMJ-PRG, F-75005, Paris, France}
\email{antoine.ducros@imj-prg.fr}
\date{\today, \clocktime}

\begin{abstract}
Nous définissons une théorie des formes de type~$(p,q)$
et de courants sur les espaces de Berkovich qui est parallèle
à la théorie des formes différentielles sur les espaces complexes:
intégration des formes de bidegré maximal,
formule de Poincaré--Lelong, 
notion de fonction plurisousharmonique,
fibrés en droites métrisés…

Cette théorie combine des considérations de géométrie
analytique non archimédienne et de géométrie tropicale locale
où elle repose sur la théorie des superformes inventée par Lagerberg.
Cela conduit aussi à une généralisation de la condition
d'équilibre en géométrie tropicale.

Les fibrés en droites métrisés dont la métrique est lisse
disposent d'une $(1,1)$-forme de Chern.
Dans le cas d'espaces analytiques propres,
l'intégration des produits de formes de Chern est
compatible avec la théorie de l'intersection.
Les fibrés en droites métrisés dont la métrique
est seulement continue et plurisousharmonique disposent
d'un $(1,1)$-courant de Chern positif. 
Par une adaptation de la théorie de Bedford--Taylor,
nous donnons sens aux produits de tels courants,
en particulier à sa puissance maximale. 
L'un des résultats principaux de ce travail
est que l'on retrouve ainsi les mesures qui avaient
été définies par le premier auteur.

Nous détaillons enfin quelques exemples, dont le cas
des variétés abéliennes, et proposons des rudiments
de géométrie d'Arakelov non archimédienne.
\end{abstract}

\begin{altabstract}
We define a theory of $(p,q)$-forms and currents on Berkovich spaces
which is parallel to the theory of differential forms on complex spaces:
integration of forms of maximal bidegree, Poincaré--Lelong formula,
notion of plurisubharmonic function, metrized line bundles…

This theory combines considerations of nonarchimedean analytic geometry
and of local tropical geometry where it builds upon the 
notion of superforms that was invented by Lagerberg.
This also leads to
a generalization of the balancing condition in tropical geometry.

Metrized line bundles whose metric is smooth give rise to a Chern $(1,1)$-form.
In the case of proper analytic spaces, the integration of products of Chern forms is compatible with intersection theory. Metrized line bundles
whose metric is only plurisubharmonic and continuous give rise
to a Chern $(1,1)$-current. Adapting the Bedford--Taylor theory,
we give a meaning to the products of such currents, in particular
to its maximal power. One of the main results of this work
is that we then recover the measures that had been defined by
the first author.

We finaly detail some examples, notably that of abelian varieties,
and propose some rudiments of non archimedean Arakelov geometry.
\end{altabstract}

\def\degtr{\operatorname{deg.tr}\nolimits}
\def\restr{\mathbin|_}
\def\pl{\ifmmode \mathrm{PL}\else\textsc{pl}\xspace\fi}
\def\plstr{\ifmmode \mathrm{PL}^{\text{\upshape str}}\else\textsc{pl}\textsuperscript{\upshape str}\xspace\fi}
\def\Gstr{G\textsuperscript{str}}
\def\rang{\operatorname{rg}}
\def\Aut{\operatorname{Aut}}
\def\div{\operatorname{div}}
\def\ddiv#1{\delta_{\div (#1)}}
\def\supp{\operatorname{supp}}
\def\vol{\operatorname{vol}}
\def\Jac{\operatorname{Jac}}
\def\MA{\operatorname{MA}}\def\CPSH{\operatorname{CPsh}}
\def\Or{\operatorname{Or}}
\def\hPic{\mathop{\widehat{\mathrm{Pic}}}}
\def\Pic{\operatorname{Pic}}

\def\id{\operatorname{id}}
\def\Hom{\operatorname{Hom}}

\def\Spec{\operatorname{Spec}}
\def\Proj{\operatorname{Proj}}
\def\Spf{\operatorname{Spf}}
\def\Card{\operatorname{Card}}
\def\Arg{\operatorname{Arg}}
\def\GL{\operatorname{GL}}
\def\Long{\operatorname{long}}
\def\Frac{\operatorname{Frac}}
\def\Gal{\operatorname{Gal}}

\def\lpar{\mathopen{(\!(}} \def\rpar{\mathclose{)\!)}}
\def\lbra{\mathopen{[\![}} \def\rbra{\mathclose{]\!]}}

\def\groth{{\mathrm G}}

\newcommand\gm[1][]{\mathbf G_{\mathrm m\csname @ifnotempty\endcsname{#1}{,#1}}}
\def\gmtrop{\gm[\trop]}
\def\R{\mathbf R}
\def\C{\mathbf C}
\def\P{\mathbf P}
\def\hr#1{{\mathscr H}(#1)}
\def\hrt#1{\widetilde{{\mathscr H}(#1)}}
\def\Q{\mathbf Q}
\def\A{\mathbf A}
\def\Z{\mathbf Z}

\def\Tau{{\mathrm T}}

\def\N{\mathbf N}
\def\zr#1#2{\P_{#1/#2}}
\def\norm#1{\mathopen\|{#1}\mathclose\|}
\def\abs#1{\mathopen|{#1}\mathclose|}
\def\Abs#1{\left|{#1}\right|}
\def\da {\mathrm d}
\newcommand{\di}[1][]{\mathop{\mathrm d'_{#1}}\nolimits}             
\newcommand{\dc}[1][]{\mathop{\mathrm d''_{#1}}\nolimits}
\def\bi{\mathop{\mathrm b'}\nolimits}             
\def\bc{\mathop{\mathrm b''}\nolimits}
\def\ddc{\mathop{\mathrm d'\mathrm d''}\nolimits}

\def\cpct{{\text{\upshape c}}}
\def\trop{{\text{\upshape trop}}}
\def\harm{{\text{\upshape harm}}}

\let\ra\rightarrow
\let\hra\hookrightarrow
\let\phi\varphi\let\eps\varepsilon
\def\resp{\emph{resp.}\xspace}
\def\cf{\emph{cf.}\nobreak\xspace}
\def\loccit{\emph{loc.\,cit.}\nobreak\xspace}

\def\df{\mathrm{def}}
\renewcommand{\vec}{\overrightarrow}
 
\def\setminus{\mathchoice
    {\mathbin{\vrule height .72ex width 1.61ex depth -.38ex}}
    {\mathbin{\vrule height .72ex width 1.61ex depth -.38ex}}
    {\mathbin{\vrule height .50ex width 0.85ex depth -.28ex}}
    {\mathbin{\vrule height .20ex width 0.570ex depth -.24ex}}
}

\let\bar\overline

\def\sq{{\mathrm{sq}}}
\def\an{{\mathrm{an}}}
\def\gen{\mathrm{gen}}
\def\Int{\operatorname{Int}}

\usepackage[pdfusetitle]{hyperref}
   
   \hypersetup{
	pdftitle={Formes différentielles réelles et courants sur les espaces de Berkovich},
	pdfauthor={Antoine Chambert-Loir, Antoine Ducros}}

\begin{document}
\selectlanguage{french}

\maketitle

\setcounter{tocdepth}1
\tableofcontents

\chapter*{Introduction}\phantomsection
\label{introduction}

Cet article propose une théorie des formes différentielles
et des courants à valeurs réelles sur un espace analytique 
au sens de Berkovich. 
Avant d'en décrire la teneur, expliquons un peu la philosophie
qui nous a guidés.

\section{Contexte et philosophie}

\subsection{}
Pour des problèmes de  géométrie arithmétique sur un corps de fonctions~$k(C)$,
où $C$  est une variété algébrique projective sur un corps~$k$,
il est courant et fructueux 
de considérer un objet  défini sur~$k(C)$ (variété, faisceau,
morphisme...) comme un germe d'objet sur~$C$
puis de le prolonger sur~$C$ tout entier.
La conjecture de Mordell sur les corps de fonctions
devient alors un énoncé sur les sections d'une surface projective
fibrée en courbes de genre au moins~$2$ au-dessus d'une courbe projective.
On peut alors utiliser tout l'attirail de la géométrie algébrique
sur cette surface, notamment la théorie de l'intersection.

Sur les corps de nombres, la théorie des schémas de Grothendieck
donne un corps à ce principe. Cependant, le spectre~$\Spec(\Z)$
de~$\Z$ est un schéma affine de dimension~$1$, 
donc l'analogue d'une courbe algébrique affine, 
et n'a pas de compactification naturelle.
Une telle compactification ne serait vraisemblablement pas
définissable uniquement par l'algèbre. 
La nullité du degré du diviseur d'une fonction
sur une courbe projective a pour analogue 
la formule du produit. Dans cette analogie,
c'est ainsi aux points fermés d'une courbe affine
que correspondent les idéaux maximaux de~$\Z$,
et le rôle des points  à l'infini sera joué par la valeur absolue usuelle,
archimédienne, de~$\Z$, un objet de nature plus analytique, ou en tout cas 
topologique.

La théorie d'Arakelov (\cite{arakelov74,faltings84,gillet-s90})
vise à pallier cette asymétrie en proposant une théorie de l'intersection
arithmétique sur les schémas réguliers, de type fini sur~$\Z$.
Dans cette théorie, les objets algébriques (cycles, fibrés vectoriels,...)
sont munis de structures supplémentaires définies par la géométrie
analytique complexe (courants de Green, métriques hermitiennes,...).
Cette théorie a donné lieu à de très beaux développements
(par exemple, un théorème de Riemann-Roch arithmétique)
et à des applications importante en géométrie diophantienne 
classique (hauteur de cycles~\cite{faltings91,bost-g-s94},
théorèmes d'équidistribution de points de petite hauteur~\cite{szpiro-u-z97},
preuve de la conjecture de Bogomolov~\cite{ullmo98,zhang98}).

\subsection{}
Dans ce travail, nous avons souhaité remettre toutes les places
sur le même plan, non pas en inventant une structure algébrique
aux places archimédiennes, mais en tâchant de définir une structure
analytique aux places ultramétriques.
Ce n'est certes pas une idée originale.

La théorie d'Arakelov non archimédienne~\cite{bloch-g-s95b}
est une première élaboration. Étant donnée une variété
algébrique~$X$ sur le corps des fractions~$K$ d'un anneau
de valuation discrète~$R$, ces auteurs travaillent non
pas avec un modèle fixé~$\mathscr X/R$, mais avec
le système projectif de tous les modèles. Pour des raisons
techniques, ils doivent se restreindre aux modèles
réguliers dont la fibre spéciale  est un diviseur à croisements
normaux stricts. Lorsque la caractéristique résiduelle de~$R$
est strictement positive, cela les oblige à admettre
l'existence d'une résolution des singularités plongée.
Indépendamment de cette restriction (le cas où $R=\C\lbra t\rbra$
est déjà très intéressant),  leur théorie reste de nature
très algébrique: formes, courants sont définis comme
des systèmes compatibles de cycles sur les différents modèles
(pour des applications convenables fournies par la théorie
de l'intersection).

Dans le cas des courbes, la thèse~\cite{thuillier2005}
d'A.~Thuillier
fournit une théorie d'Arakelov analytique à toutes les places
fondée sur les espaces de Berkovich aux places
ultramétriques et sur l'analyse complexe aux places
archimédiennes.
Dans le cas
d'une courbe sur un corps ultramétrique, il y développe aussi
une théorie du potentiel formellement analogue à celle
qui vaut sur les surfaces de Riemann.
Le cas de la droite projective a été développé indépendamment
par~\cite{baker-rumely2006,baker-rumely2010}.
Ces considérations recoupent aussi les travaux
de~\cite{rivera-letelier2003,favre-rivera-letelier2006,favre-rl2007}
en dynamique non archimédienne.

En dimension arbitraire,
la théorie des mesures sur les espaces de Berkovich 
(due au premier auteur~\cite{chambert-loir2006} et raffinée
par~\cite{gubler2007a}) fournit une troisième incarnation de
cette idée. Si $X$ est propre, de dimension~$n$, et si
$\bar L$ est un fibré en droites sur~$X$ muni d'une métrique admissible
au sens de~\cite{zhang95b}, on construit une mesure
sur l'espace analytique $X^\an$ associé à~$X/K$ par Berkovich
que nous notons $c_1(\bar L)^n$.
En effet, cette mesure est un analogue de la mesure
de même nom lorsque $\bar L$ est un fibré en droites
hermitien sur une variété analytique compacte de dimension~$n$.
Sa masse totale est égale au degré $\deg(c_1(L)^n|X)$ de~$X$ pour~$L$.
Elle donne aussi lieu à des théorèmes d'équidistribution
des points de petite hauteur tout à fait analogues
à ceux évoqués plus haut dans le cas archimédien
(\cite{chambert-loir2006,baker-rumely2006,favre-rivera-letelier2006,gubler2007b,yuan2008}).
Voir aussi~\cite{chambert-loir2011} pour un survol
de la théorie qui met en parallèle les cas archimédien et ultramétrique.
Dans~\cite{gubler2007b}, Gubler montre comment
les idées de~\cite{ullmo98,zhang98} se transposent
et permettent de prouver la conjecture de Bogomolov
sur les corps de fonctions lorsque la variété abélienne
sous-jacente est {\og totalement dégénérée\fg}.

Divers mathématiciens se sont aussi intéressé à l'analogue
du théorème de Yau (conjecture de Calabi) dans ce contexte.
Citons un travail non publié de Kontsevich et Tschinkel (2000).
Plus récemment,
l'unicité d'une métrique conduisant à une mesure
prescrite (théorème de Calabi)  est prouvée 
dans~\cite{yuan-zhang2009}
en adaptant la preuve de~\cite{blocki2003}.
L'article~\cite{liu2011} construit quelques solutions
dans le cas de variétés abéliennes totalement dégénérées.
Les articles~\cite{boucksom-favre-jonsson2012a,boucksom-favre-jonsson2012b}
prouvent l'existence de solutions dans de nombreux cas algébriques.

\section{Cet article}

\subsection{}
Dans le cas d'un fibré en droites hermitien~$\bar L$
sur une variété complexe~$X$  $c_1(\bar L)$ est une forme 
différentielle de type~$(1,1)$ et sa puissance
$c_1(\bar L)^n$ est une forme de  type~$(n,n)$,
laquelle donne lieu à une mesure si $X$ est de dimension~$n$.
La théorie des mesures de~\cite{chambert-loir2006}
avait ceci d'insatisfaisant qu'elle définissait
une mesure $c_1(\bar L)^n$ tout en s'avérant incapable
de donner un sens à l'objet $c_1(\bar L)$,
lorsque $\bar L$ est un fibré métrisé 
sur une variété algébrique propre de dimension~$n$
sur un corps valué complet.

Combler cette frustration  est la raison d'être du présent travail.

Au passage, nous étendons la théorie en traitant
des variétés analytiques ni nécessairement algébriques,
ni nécessairement compactes. Le caractère local
des mesures $c_1(\bar L)^n$ (établi dans~\cite{chambert-loir2011})
apparaîtra ainsi bien plus naturellement.

Soit donc $k$ un corps valué complet, non archimédien.
Soit $X$ un espace $k$-analytique au sens de Berkovich~\cite{berkovich1990,berkovich1993}.

\subsection{Fonctions et formes lisses}
Nous appelons \emph{fonction lisse} sur~$X$ une fonction 
sur~$X$ à valeurs réelles qui s'écrit
localement $\phi(\log\abs{f_1},\dots,\log\abs{f_r})$,
où $f_1,\dots,f_r\in\mathscr O_X^\times$ 
sont des fonctions holomorphes inversibles sur~$X$
et $\phi$ est une fonction~$\mathscr C^\infty$ sur~$\R^r$ à valeurs
réelles. Les fonctions lisses forment un sous-faisceau~$\mathscr A_X$
du faisceau des fonctions continues sur~$X$ à valeurs réelles.
Si $X$ est bon et paracompact,
par exemple si $X$ est l'analytifié d'une variété algébrique, 
ce faisceau est fin.

Plus généralement, si $p$ et~$q$ sont deux entiers naturels,
nous définissons  une notion de forme différentielle lisse de type~$(p,q)$
sur~$X$, qui redonne la notion de fonction lisse pour  $(p,q)=(0,0)$.
Ce sont naturellement les sections
de faisceaux~$\mathscr A_X^{p,q}$ qui sont nuls 
si $p>\dim(X)$ ou $q>\dim(X)$. 
La somme directe de ces faisceaux est une algèbre bigraduée commutative; 
elle est munie d'opérateurs différentiels~$\di$ et~$\dc$ 
de degrés~$(1,0)$ et~$(0,1)$  et de carrés nuls.

Si $\dim(X)=n$, on peut intégrer les formes de type~$(n,n)$,
pourvu bien entendu qu'elles soient intégrables, ce qui est
par exemple le cas si leur support  est compact.
Si $X$ est l'analytifié d'une variété propre, 
l'intégrale d'une forme sur~$X$
de la forme~$\di\omega$ est nulle. 

Plus généralement, Berkovich a introduit la notion de \emph{bord}~$\partial(X)$
de l'espace analytique~$X$ et nous
définissons aussi l'intégrale d'une forme de type~$(n-1,n)$
sur ce bord. Ces définitions donnent lieu à l'analogue de
la \emph{formule de Stokes}:
\[ \int_X\di\omega = \int_{\partial(X)}\omega, \]
si $\omega$ est une forme à support compact de type~$(n-1,n)$ sur~$X$.

\subsection{Courants}
Nous introduisons alors la notion de courant. Comme
chez Schwartz et De Rham, ceux-ci sont définis comme
les formes linéaires sur l'espace des formes différentielles lisses
à support compact (et, techniquement, qui évitent le bord)
qui vérifient une condition de continuité convenable.
Par l'intégration des formes différentielles de type~$(n,n)$
et le produit extérieur, les formes s'identifient
à des courants.
L'intégration sur~$X$, sur un sous-espace de~$X$,
une mesure sur~$\Int(X)$,
donnent aussi lieu à des courants.  

Si $f$ est une fonction méromorphe régulière sur~$X$
(c'est-à-dire définie sur un ouvert de Zariski dense),
la fonction $\log\abs f$ définit un courant.
Notons $\div(f)$ le diviseur de~$f$ (un cycle de codimension~$1$
sur~$X$) et $\delta_{\div(f)}$ le courant sur~$X$
qui se déduit par linéarité des divers courants d'intégration
sur les composantes de~$\div(f)$.
Nous prouvons alors l'analogue de la \emph{formule de Poincaré-Lelong}:
\[ \ddc \log\abs f = \delta_{\div(f)}. \]

\subsection{}
Étant donné un  fibré en droites~$L$ sur~$X$,
on définit alors naturellement la notion de métrique lisse 
sur~$L$.
Grâce à la formule de Poincaré-Lelong,
un tel fibré métrisé~$\bar L$ possède une forme de courbure $c_1(\bar L)$,
définie de façon analogue à la forme de courbure
d'un fibré en droites hermitien sur une variété complexe.

Si $X$ est l'analytifié d'une variété algébrique propre de dimension~$n$,
on démontre alors à l'aide de la formule de Stokes
que l'intégrale sur~$X$ de la $(n,n)$-forme lisse $c_1(\bar L)^n$
est égal au degré d'intersection $\deg(c_1(L)^n\cap[X])$.

D'autre part, un modèle formel~$(\mathfrak X,\mathfrak L)$ de~$(X,L)$ 
sur~$k^\circ$ fournit naturellement une métrique continue 
(qualifiée de \emph{formelle}) sur~$L$.
Pour une telle métrique, la norme d'une section locale de~$L$
est égale au maximum d'une famille finie de valeurs absolues de 
certaines fonctions
holomorphes inversibles. Ainsi,  et contrairement à l'intuition véhiculée
par les articles précédents, une métrique formelle  n'est \emph{pas}  lisse
en général.  

Cela fournit un point de contact entre la théorie
du présent article et le mémoire~\cite{maillot2000}
sur la théorie d'Arakelov des variétés toriques.
Ce dernier travail étudie en détail  
les métriques hermitiennes naturelles sur les fibrés
en droites hermitiens sur une variété torique complexe --- définies
précisément par des maximum de valeurs absolues de fonctions holomorphes.
Un tel  fibré en droite hermitien possède une courbure, laquelle
n'est plus une forme différentielle mais un courant.
Un outil crucial de~\cite{maillot2000} est fourni
par la théorie de~\cite{bedford-t82,demailly1985} qui permet
de les multiplier. (Rappelons que les courants
ne sont en général pas plus multipliables que ne le sont les distributions.)
Le fait essentiel
est que la fonction $\max(x_1,\dots,x_r)$ sur~$\R^r$
est limite uniforme de fonctions convexes lisses.

\subsection{}
Nous définissons ainsi des notions de formes et courants positifs.
Cela nous amène à développer les rudiments de la théorie
des fonctions plurisousharmoniques ---
une fonction lisse~$u$ est dite plurisousharmonique si  
$\ddc u$ est un courant positif.
En fait, les seules fonctions plurisousharmoniques
que nous considérerons dans cet article
sont des limites uniformes de fonctions lisses plurisousharmoniques.
En copiant les arguments de~\cite{bedford-t82,demailly1985},
nous pouvons définir des produits de courants
$T\wedge \ddc u_1\wedge\dots\wedge\ddc u_r$, où $T$
est un courant positif et $u_1,\dots,u_r$ des fonctions
plurisousharmoniques comme ci-dessus.  

Grâce à cette théorie, on peut multiplier les courants de courbure
des fibrés en droites munis de métriques formelles.
Nous prouvons alors que la théorie développée ici redonne
les mesures définies dans~\cite{chambert-loir2006}
qui correspondent au cas où le nombre de fibrés en droites
est égal à la dimension.

La théorie construite dans ce dernier article partait
du cas des métriques formelles, alors considérées comme métriques
{\og lisses\fg}. La théorie de l'intersection sur la fibre spéciale
du modèle formel fournissait des mesures atomiques sur~$X$ ;
un argument d'approximation
permettait de construire des mesures  pour des fibrés en droites
munis de métriques plus générales.
Dans le présent article, le point de vue est renversé : les
métriques formelles ne sont pas lisses, mais limites
uniformes de métriques lisses convenables, et les mesures
associées ne sont définies que via un passage à la limite.

\subsection{}
Notre théorie de l'intégration des formes offre d'autres
analogies avec la théorie complexe.
Par exemple, nous démontrons un théorème de continuité (et de lissité)
des intégrales fibre, analogue au théorème de Stoll~\cite{Stoll-1966a}.

Nous démontrons aussi des analogues des formules 
de Levine et de Bochner--Martinelli en théorie des courants.

Mentionnons également l'article~\cite{DucrosHrushovskiLoeser-2023}
où des limites de familles à un paramètre d'intégrales complexes
sont interprétées comme des intégrales au sens de travail
sur des espaces analytiques sur le corps $\C\lpar T\rpar$.

\subsection{}
Dans un dernier chapitre, nous appliquons notre théorie au cas des
variétés abéliennes au moyen de la théorie de l'uniformisation de 
Tate--Raynaud. Nous décrivons les métriques canoniques sur
les fibrés en droite et les mesures de Monge--Ampère associées.


\subsection{}
La construction des faisceaux de formes différentielles
repose sur un travail
récent d'A.~Lagerberg~\cite{lagerberg2012}
dont le but est d'adapter 
les outils classiques de la théorie du pluripotentiel sur
les variétés complexes
à l'analyse convexe sur des espaces vectoriels réels.

Pour tout ouvert~$U$ de l'espace vectoriel~$\R^n$, Lagerberg construit
une algèbre bigraduée commutative~$\mathscr A^{*,*}(U)$
en dédoublant le complexe de De Rham de~$\R^n$: pour tous~$(p,q)\in\N^2$,
\[ \mathscr A^{p,q}(U) = \mathscr C ^\infty(U) \otimes \bigwedge^p (\R^n)^* \otimes \bigwedge^q (\R^n)^*. \]
En particulier, $\mathscr A^{0,0}(U)$ est l'algèbre
des fonctions~$\mathscr C^\infty$ sur~$U$.
Pour $i\in\{1,\dots,n\}$, on note $\di x_i$ 
et $\dc x_i$  la $i$-ieme forme coordonnée vue respectivement
dans~$\mathscr A^{1,0}$ et $\mathscr A^{0,1}$.
Cette algèbre possède deux dérivations, notées~$\di$ et~$\dc$,
analogues des dérivations holomorphes et antiholomorphes.
Elles sont caractérisées par la formule
\[  \di f = \sum_{i=1}^n \frac{\partial f}{\partial x_i} \di x_i
\quad\text{et}\quad
  \dc f = \sum_{i=1}^n \frac{\partial f}{\partial x_i} \dc x_i, \]
si $f$ est une fonction~$\mathscr C^\infty$ sur~$U$.
Cette algèbre possède aussi une involution qui échange~$\di$ et~$\dc$.

L'opérateur $(\ddc f)^n$  est l'analogue de l'opérateur
de Monge-Ampère complexe; il s'exprime d'ailleurs en
termes de l'opérateur de Monge-Ampère réel: on a en effet 
\[ (\ddc f)^n  = n! \det\left( \big( \frac{\partial^2 f}{\partial x_i\partial x_j}\big) \right)  \di x_1\wedge\dc x_1 \wedge\dots\wedge\di x_n
\wedge\dc x_n. \]

Les éléments de $\mathscr A^{p,q}(U)$
sont appelées \emph{superformes} de type~$(p,q)$ sur~$U$;
le préfixe super est ajouté par référence à la superanalyse de Berezin;
comme nous n'utiliserons pas cette dernière, nous nous permettrons
souvent d'omettre ce préfixe.

Une $(n,n)$-forme~$\omega$ sur un ouvert~$U$ de~$\R^n$
s'écrit
\[ \omega=\omega^\sharp \di x_1\wedge\dc x_1 \wedge\dots\wedge\di x_n
\wedge\dc x_n, \]
où $\omega^\sharp$ est une fonction~$\mathscr C^\infty$.
Son intégrale est alors définie par la formule
\[ \int_U \omega = \int_U \omega^\sharp(x_1,\dots,x_n)\,\mathrm dx_1\dots\mathrm dx_n. \]
Cela donne automatiquement naissance à une théorie des courants.
Il y a aussi une notion de forme positive, de courant positif;
les fonctions convexes sont caractérisées par la positivité
de leur image par l'opérateur~$\ddc$, etc.

\subsection{}
Le lien entre espaces de Berkovich et espaces réels
est fourni par la {\og géométrie tropicale\fg}.

Le point de vue tropical en géométrie analytique sur un corps valué~$k$
(ultramétrique ou non)
consiste à approcher un espace analytique~$X$
par ses \emph{tropicalisations}, ou \emph{amibes},
images de~$X$ dans~$\R^n$ par des applications de la forme
$x\mapsto (\log\abs{f_1(x)},\dots,\log\abs{f_n(x)})$,
où $f_1,\dots,f_n$ sont des fonctions analytiques
inversibles sur~$X$. Introduit initialement dans le cas où $k$
est le corps des nombres complexes et $X$ est une variété algébrique
(voir \cite{bergman1971,viro2001,mikhalkin2006}),
ce point de vue a ensuite été développé dans le cas
ultramétrique 
(citons notamment~\cite{bieri-groves1984,einsiedler-kapranov-lind2006,gubler2007a}).

Il est fondamental que les amibes 
sont alors des sous-espaces linéaires par morceaux de~$\R^n$,
connexes si $X$ est connexe.
Ce fait,
initialement démontré dans~\cite{bieri-groves1984},
a ensuite reprouvé par~\cite{gubler2007a} à l'aide
d'un argument de Berkovich issu de son analyse
locale~\cite{berkovich2004} des espaces analytiques
(qui repose sur l'existence d'altérations~\cite{dejong96}),
puis dans~\cite{ducros2012b} comme conséquence de l'élimination
des quantificateurs pour les corps valués algébriquement clos.
Voir aussi l'article~\cite{payne2009} pour
une  autre incarnation du lien 
intime entre géométrie analytique et géométrie tropicale.

\subsection{}
Plaçons nous pour commencer en géométrie analytique complexe.
Soit $X$ une variété analytique complexe, $f_1,\dots,f_n$
des fonctions holomorphes inversibles sur~$X$
et $f=(f_1,\dots,f_n)\colon X\ra (\C^*)^n$ l'application qui s'en déduit.
Notons $f_\trop\colon X\ra\R^n$ la composition de~$f$
et de l'application $(\C^*)^n\ra \R^n$ donnée par $(z_1,\dots,z_n)\mapsto
(\log\abs{u_1},\dots,\log\abs{u_n})$.

Par les formules
\[ f^*u=u\circ f_\trop, \quad f^*\di x_i = \mathrm d \log\abs{f_i} 
\quad \text{et}\quad 
f^*\dc x_i = \mathrm d \Arg(f_i) , \]
une superforme~$\omega$  de type~$(p,q)$ sur~$\R^n$ donne lieu
à une forme différentielle réelle~$f^*\omega$ de degré~$p+q$ sur~$X$.
On prendra garde que $f^*\omega$
n'est pas forcément de type~$(p,q)$ : cela se voit déjà
sur $f^*\di x_i$ qui contient une composante holomorphe
et une composante antiholomorphe.
En revanche, si $\omega$ est de type~$(p,p)$ et symétrique,
$f^*\omega$ est de type~$(p,p)$.

\subsection{}
Revenons maintenant au cas des espaces analytiques sur un corps
ultramétrique~$k$.
Soit $X$ un tel espace. La définition d'une forme lisse
de type~$(p,q)$ est définie en considérant des tropicalisations
locales $f\colon X\ra\gm^r$ et des superformes sur l'image~$f_\trop(X)$
dans~$\R^r$,
une fois adaptée la définition de Lagerberg au cas des
espaces linéaires par morceaux.

En revanche, si $X$ est de dimension~$n$,
on voudrait définir l'intégrale d'une forme de type~$(n,n)$
de la forme $f^*\omega$ comme la somme des intégrales
à la Lagerberg sur les sous-espaces de dimension~$n$
de $f_\trop(X)$.
Cela pose deux problèmes supplémentaires: 
\begin{itemize}
\item L'intégrale définie par Lagerberg
dépend du  système de coordonnées canonique sur~$\R^n$
et ses sous-espaces linéaires n'en disposent pas a priori;
\item Si l'on pense à la formule de projection
en théorie de l'intersection, ou à la formule du changement
de variables dans les intégrales multiples,
l'intégrale de~$f^*\omega$ ne peut
être égale à $\int\omega$ qu'à un facteur correctif près,
espèce de degré de~$f$.
\end{itemize}
Nous les résolvons simultanément en introduisant
la notion de \emph{calibrage}~$\mu$ d'un sous-espace linéaire
par morceaux de dimension~$n$ de~$\R^r$. À des questions d'orientation près,
il s'agit essentiellement d'une famille de $n$-vecteurs définis face par face.
En contractant la superforme~$\omega$ de type~$(n,n)$ par ces $n$-vecteurs,
on la transforme en une $n$-forme différentielle $\langle\omega,\mu\rangle$
sur l'espace linéaire par morceaux, forme qu'on intègre
de façon usuelle. La théorie de Lagerberg
correspond au choix du $n$-vecteur $\frac\partial{\partial x_1}\wedge
\dots\frac\partial{\partial x_n}$ sur~$\R^n$.

Si $X$ est compact, nous définissons alors un calibrage canonique~$\mu_f$
de $f_\trop(X)$. Négligeant ici des problèmes de bord,
le fait que, lorsque $r=n$, $f$ soit \emph{fini et plat}
au-dessus du squelette de~$\gm^n$, et le degré de~$f$ en ces points,
intervient crucialement dans cette construction.

Signalons qu'il aurait peut-être été possible d'éviter d'introduire
cette notion de calibrage en utilisant des structures $\Z$-linéaires
sur les espaces linéaires par morceaux considérés
mais nous n'avons pas développé cette approche.

\subsection{}
Une variante de ces constructions conduit à l'intégrale sur le bord de~$X$
d'une forme de type~$(n-1,n)$
et la formule de Stokes se déduit de celle  en géométrie différentielle.
Si la définition de l'intégrale de bord
considère les bords des tropicalisations,
un avatar de la condition d'équilibre en géométrie tropicale
permet de prouver que la terminologie est justifiée,
au sens où l'intégrale sur le bord de~$X$ d'une forme
nulle au voisinage du bord est nulle.

\subsection{}
En fait, la description des paragraphes précédents correspond
à la construction mise au point dans une première version de ce texte
et il s'est avéré commode de calculer les intégrales sur l'espace~$X$
lui-même plutôt que sur ses tropicalisations.

Plus précisément, un espace analytique est muni d'une classe
de parties linéaires par morceaux naturelles, que nous appelons
parties squelettiques. Ce sont ces dernières que nous calibrons
et sur lesquelles les intégrales se calculent, au moyen de
cartes affines.

Un intérêt de cette approche est qu'elle permet de faire
des raisonnement non seulement locaux, mais également G-locaux,
sur l'espace analytique~$X$. Au passage, cela nous a conduit
à développer une théorie des fonctions G-lisses et des G-formes;
les fonctions G-lisses sont les analogues des fonctions lisses par morceaux
sur une variété réelle ou complexe.

\subsection{}
Sur la base de  la première version de ce texte,
W.~Gubler et K.~Künnemann ont développé une théorie  voisine
des $\delta$-formes qui, dans notre langage,
donnent lieu à des courants d'un certain type,
qui combinent nos formes et des cycles tropicaux sur l'espace
affine d'une tropicalisation.
La construction du produit des $\delta$-lisses repose
alors sur le produit des formes lisses et sur la
théorie de l'intersection tropicale.

Nous avons ainsi été conduits à réinterpréter la théorie de l'intersection
tropicale dans notre cadre, par des méthodes à la Bedford--Taylor.
(A.~Mihatsch~\cite{Mihatsch-2023} a également proposé une telle interprétation.)

Cela nous a amené à considérer le produit
d'un espace analytique au sens de Berkovich et d'un espace affine réel,
et les formes et courants sur celui-ci.
Il a donc fallu munir ce produit d'une structure convenable,
que nous avons appelée \emph{espace tropical}: ce sont
des espaces munis de classes convenables de  parties squelettiques
et de  tropicalisations, les deux exemples extrêmes étant les espaces
analytiques et les espaces linéaires par morceaux.

De fait, toute la théorie précédente est menée 
d'abord dans ce cadre abstrait qui occupe la première partie du livre.
A posteriori, il nous semble que ce partage permet de distinguer clairement
la part de la géométrie linéaire par morceaux et celle 
de la géométrie algébrique/analytique  dans notre théorie
des formes et courants.

\subsection{}
Nous nous sommes efforcés de rédiger le texte
dans la généralité maximale compatible avec nos compétences
techniques. Ainsi, on ne se limite pas aux espaces analytiques algébriques,
on ne les suppose pas lisses, ni même réduits. En fait, notre théorie est très {\og générique\fg}; par exemple, les supports
des formes lisses de degré maximal ne rencontrent aucun
fermé de Zariski d'intérieur vide, en particulier aucun point rigide
(en dimension~$>0$).  Ainsi,
des hypothèses de lissité, etc. n'interviendraient jamais
de façon cruciale au cours des preuves.

La théorie est valable en toute caractéristique.
Elle n'utilise à aucun moment la résolution des singularités
(Hironaka, Temkin) ni même l'existence d'altérations (De Jong), 
notamment grâce à l'article~\cite{ducros2012b} du second auteur 
qui permet de dégager, via la théorie des modèles,
les propriétés des  tropicalisations et des squelettes.
Nous n'utilisons pas de modèles formels pour développer
la théorie (mais ils fournissent une source importante d'exemples).
Toutefois, nous faisons appel à des arguments de réduction, proches de ceux
permis par la géométrie formelle, mais de nature plus locale
et techniquement plus aisés, 
fournis par la théorie de Temkin~\cite{temkin2000,temkin2004}.

La théorie est aussi valable en valuation triviale.
C'est d'ailleurs un point crucial pour le calcul des
opérateurs de Monge--Ampère associés aux métriques formelles
qui repose sur la théorie appliquée au corps résiduel  du corps de base,
corps résiduel que l'on munit de la valeur absolue triviale.

\section{Prolongements}

Depuis la parution de la première version de ce texte,
de nombreux travaux en ont utilisé les définitions et les méthodes,
en ont proposé des variantes et des approfondissements.

Mentionnons ici quelques uns, qui nous paraissent les plus significatifs,
en présentant par avance nos excuses pour les éventuelles omissions.

\subsection{}
Pour l'opérateur différentiel~$\dc$,
les formes lisses sur un espace analytique~$X$
donnent lieu à un analogue du complexe de Dolbeault.
Dans le cas algébrique,
\cite{Jell-2016}, Jell a démontré que ce complexe fournit
une résolution du faisceau constant~$\R$,
et plus généralement du faisceau des~$(p,0)$-formes.
Sa preuve s'étend en général.
Notons que ce complexe est de longueur~$\dim(X)$,
de façon compatible au fait que la dimension topologique de~$X$
est au plus égale à sa dimension.

Jell~\cite{Jell-2019} a démontré que, même pour une courbe propre, 
la cohomologie de ces complexes de Dolbeault peut être de dimension infinie.
Jell et Wanner~\cite{JellWanner-2018} ont observé 
que ce n'est néanmoins pas le cas pour les courbes de Mumford,
c'est-à-dire celles dont la réduction est maximalement dégénérée;
dans ce cas, il y a même un analogue de la dualité de Poincaré.

Pour remédier à ce défaut, Gubler, Jell et 
Rabinoff~\cite{GublerJellRabinoff-2021} ont construit une variante 
de notre théorie des formes lisses dans laquelles les tropicalisations
sont, non pas seulement de la forme $\log \abs f$ pour une fonction
analytique inversible~$f$, mais des « fonctions harmoniques »
dans un sens permis par la théorie d'intersection
sur le corps résiduel. Nous avons réinterprété une partie de ce travail
sous le vocable de « G-formes harmoniquement tropicalisées »;
comme déjà dans remarqué dans~\cite{GublerJellRabinoff-2021},
de telles formes apparaissent naturellement dans la théorie
des variétés abéliennes lorsque leur extension de Raynaud  n'est
pas isotriviale.

Dans~\cite{Liu-2020}, Liu construit la classe de Chern
d'un cycle de codimension~$p$ sur une variété algébrique propre~$X$
à valeurs dans le groupe $H^{p,p}(X)$ de cohomologie de Dolbeault.
Pour une variété définie sur~$\Q$, il démontre
en particulier qu'un cycle (défini sur~$\Q$) dont la classe de cohomologie 
de de Rham 
complexe usuelle  est nulle 
a également une classe nulle pour toutes les cohomologie de Dolbeault
non archimédiennes associées aux différentes places finies de~$\Q$.

La cohomologie de ce complexe de Dolbeault est très certainement liée aux
groupes d'homologie tropicaux de~\cite{ItenbergKatzarkovMikhalkinEtAl-2019},
mais nous n'avons pas d'énoncé formel précis à proposer.

\subsection{}
Nous avons déjà mentionné la théorie des $\delta$-formes
de Gubler et Künnemann~\cite{GublerKunnemann-2017}.
Signalons l'article de Mihatsch~\cite{Mihatsch-2024}
dans lequel il raffine leur construction et, surtout, 
en donne une application à la théorie de l'intersection sur
les tours de Lubin--Tate.

\subsection{}
La résolution par Yau du problème
de Monge--Ampère fut un progrès majeur
de la géométrie complexe. L'analogue non archimédien
a été proposé par Kontsevich et Tschinkel
et résolu par Boucksom, Favre et Jonsson
dans \cite{BoucksomFavreJonsson-2015, BoucksomFavreJonsson-2016}.
Dans l'esprit,
leur théorie est liée à celle de ce texte mais les définitions
et les méthodes sont différentes, car elles reposent d'emblée sur
la considération de modèles et de la théorie de l'intersection sur
leur fibre spéciale.

\subsection{}
Dans~\cite{BoucksomEriksson-2021},
Boucksom et Eriksson construisent une variante
de l'accouplement de Deligne dans notre théorie.
Cet accouplement est également considéré dans l'article~\cite{YuanZhang-2025}
de Yuan et Zhang.

\subsection{}
La notion de forme et de courant positifs est fondamentale 
dans le développement de notre théorie, où une notion convenable
de fonction psh joue un rôle crucial. 
La définition que nous proposons impose la possibilité
d'approximation d'une fonction psh par des fonctions lisses psh.
Gubler et Rabinoff ont considéré une définition différente,
fondée sur la restriction aux courbes et la théorie de Thuillier
des fonctions psh sur les courbes. 
Ils établissent notamment un principe du maximum.
Nous ne savons pas si un tel principe vaut pour les fonctions psh
de ce texte.

Notre notion de fonction psh a un lien étroit avec celle
de métrique semipositive au sens de Zhang~\cite{zhang95b}.
Signalons l'article~\cite{GublerMartin-2019} qui étudie 
plus en détail cette dernière.

Signalons enfin l'article~\cite{BoteroBurgosGilSombra-2022} dans lequel
Botero, Burgos Gil et Sombra étudient la notion de fonction
convexe sur un espace linéaire par morceaux.

\section{Perspectives}

\subsection{}
Qu'est-ce qu'une forme de Kähler dans notre contexte et y
a-t-il un analogue du théorème de plongement de Kodaira ?
La difficulté vient du fait que le support d'une $(n,n)$-forme
(qui serait la forme associée à la forme de Kähler) est
un squelette et donc très petit par rapport à l'espace.
L'analogue non archimédien de la métrique de Fubini est certes
singulière, mais son courant de courbure est concentré en un point
de la droite projective. Dans le cas d'un espace projectif
de dimension plus grande, la « forme volume » associée à cette
métrique est également concentrée en un point.

\subsection{}
Il serait intéressant de développer une théorie sur~$\Z$,
dans le cadre des espaces considérés par Poineau dans~\cite{poineau2012}.
Mentionnons son étude~\cite{Poineau-2024}
de la variation des mesures d'équilibre
en dynamique algébrique.
Comme notre théorie est d'essence ultramétrique, il faudra sans doute
des arguments nouveaux pour traiter les places archimédiennes.

Il serait également intéressant d'aborder la dynamique non archimédienne
au moyen des concepts fournis par le présent travail, 
de même que l'analyse complexe est un outil essentiel en dynamique 
complexe.

\subsection{}
Le développement d'une présentation \emph{analytique} de la théorie
d'Arakelov non archimédienne était l'une des motivations de ce texte.
Même si des premiers résultats  ont été donnés Gubler et Künnemann~\cite{GublerKunnemann-2017}, puis \cite{Mihatsch-2024} qui prouve qu'un cycle intersection
complète possède un courant de Green,
ce développement est encore inabouti.
Un analogue de la {\og formule de Mahler\fg}, 
dans l'esprit de~\cite{maillot2000,chambert-loir-thuillier2009}
a été démontré dans~\cite{BoucksomEriksson-2021},
par une variante de la méthode de~\cite{chambert-loir-thuillier2009}.
Comme le suggère le travail de Mihatsch~\cite{Mihatsch-2024},
il est vraisemblable qu'un tel développement aura des conséquences
sur la géométrie arithmétiques des variétés de Shimura.

\section{Remerciements}

C'est en janvier~2010 que nous avons débuté ce travail,
au cours d'une visite d'A.~D. à A.~C.-L.,
alors que ce dernier passait l'année à l'Institute for Advanced Study
de Princeton. Nous voulons remercier l'Institute pour son hospitalité
et l'atmosphère qui nous a permis  d'entamer un projet
qui semblait alors bien hasardeux.
A.~C.-L. a bénéficié du soutien de la National Science
Foundation, agrément No.~DMS-0635607, des largesses de l'Institut
universitaire de France. Il a  été membre
des projets~\emph{Positive}, \emph{Defigeo}
de l'Agence nationale de la recherche.
A.~D. a bénéficié du soutien de l'Institut universitaire de France,
de l'hospitalité du MSRI et de l'institut Weizmann;
il a été membre des projets~\emph{BERKO},
\emph{Valcomo}, \emph{Defigeo} de l'Agence nationale de la recherche.
Les deux auteurs sont membres du projet~\emph{AdAnAr} de cette agence.

Le contenu de cet article a été exposé dans diverses conférences et séminaires,
\emph{Algebraic cycles and L-functions,} Regensburg (Allemagne);
\emph{Number Theory Days,} Zürich (Suisse);
\emph{Séminaire de géométrie algébrique},
\emph{Séminaire de géométrie tropicale}, 
\emph{Séminaire de géométrie et théorie des modèles}, Paris,
plusieurs sessions des \emph{Nonarchimedean Days} à Regensburg (Allemagne),
ainsi que dans une série de trois Simons symposiums;
il a fait l'objet d'un cours à l'université de Westlake, Hangzhou (Chine).

Nous remercions
les organisateurs de ces séminaires ainsi que leur public 
et d'autres collègues pour l'intérêt qu'ils nous ont manifesté,
en particulier M.~Baker, V.~Berkovich, J.-B.~Bost, S.~Boucksom,  
J.-P.~Demailly, P.~Dolbeault, 
C.~Favre, W.~Gubler, E.~Hrushovski, N.~Hultberg, K.~Künnemann, I.~Itenberg,
M.~Jonsson, F.~Loeser, M.~Maculan, A.~Mihatsch, C.~Peskine, J.~Poineau, 
M.~Temkin, A.~Thuillier, C.~Voisin, A.~Werner, X.~Yuan et S.~Zhang.

Pour terminer, nous présentons nos excuses pour le délai
entre la première version (avril 2012) et celle-ci (juillet 2025)
et pour la multiplication par cinq du nombre de pages.
Nous tâcherons qu'il n'en soit pas de même pour la version définitive.

\part{Géométrie tropicale}

\def\pf{\text{\upshape pf}}
\def\Rad{\text{\upshape Rad}}
\def\Bor{\text{\upshape Bor}}
\def\Lunloc{{\mathrm L^1_{\mathrm{loc}}}}
\def\pr{\operatorname{pr}}
\def\MAT{\mathrm M}
\def\ShHom{\mathop{\mathscr H\!\mathit{om}}\nolimits}

\chapter{Superformes sur un espace vectoriel réel (d'après Lagerberg)}
\label{sec.lagerberg}

\section{Superformes sur un espace vectoriel réel}

La construction algébrique suivante est à la base de cet article
elle a été introduite par Lagerberg dans~\cite{lagerberg2012}.
Ce sont les versions ponctuelles des espaces de formes
qui nous intéresseront plus loin.

\begin{defi}
Soit $V$ un espace vectoriel réel de dimension finie.
On note~$A^{*,*}(V)$ l'algèbre bigraduée commutative
somme directe, pour tous $p,q\geq 0$, des espaces vectoriels
$A^{p,q}(V)=\Lambda^pV^*\otimes\Lambda^qV^*$.
\end{defi}
On dit qu'un élément de $A^{p,q}(V)$ est
une \emph{superforme}
de type~$(p,q)$ sur~$V$. 
Ces espaces $A^{p,q}(V)$ sont bien sûr nuls 
si $p>\dim(V)$ ou $q>\dim(V)$.

D'après la formule usuelle calculant l'algèbre extérieure
d'une somme directe, 
$A^{*,*}(V)$ est l'algèbre extérieure de l'espace vectoriel~$V^*\oplus V^*$.
La condition de commutativité sur cette algèbre est ainsi la suivante: 
si $\omega$ est de type~$(p,q)$ et $\omega'$ de 
type~$(p',q')$,
alors
\[ \omega' \wedge \omega = (-1)^{(p+q)(p'+q')} \omega\wedge\omega'. \]

Pour distinguer les formes linéaires apparaissant
dans le premier facteur~$\Lambda^p \vec V^*$ de celles apparaissant
dans le second~$\Lambda^q \vec V^*$, on convient de noter
systématiquement les premières avec le préfixe~$\di$
et les secondes avec le préfixe~$\dc$. 
Ainsi, si $x\in V^*$, on note $\di x$ l'élément~$x\otimes 1$ dans~$A^{1,0}(V)$
et $\dc x$ l'élément~$1\otimes x$ de $A^{0,1}(V)$.

Soit $(x_1,\dots,x_n)$ une base de~$V^*$.
Pour tout couple $(I,J)$ de suites strictement croissantes de~$\{1,\dots,n\}$,
où $I=(i_1,\dots,i_p)$ et $J=(j_1,\dots,j_q)$,
notons $\di x_I\wedge \dc x_J$ la superforme de type~$(p,q)$ définie par
\[  \di x_I\wedge \dc x_J = \di x_{i_1}\wedge\dots \wedge \di x_{i_p}
 \wedge \dc x_{j_1}\wedge\dots \wedge \dc x_{j_q}. \]
Ces superformes $\di x_I\wedge\dc x_J$ fournissent une base
de~$A^{p,q}(V)$.

Toute application linéaire~$f$ donne lieu à 
un morphisme d'algèbres évident $f^*\colon A^{*,*}(V)\to A^{*,*}(W)$.

\subsection{}
Il existe un unique automorphisme d'algèbres,
$\mathrm J\colon A^{*,*}(V)\to A ^{*,*}(V)$,
tel que $\mathrm J(\di x_i)=\dc x$ et $\mathrm J(\dc x)=-\di x$
pour tout $x\in V$. 

On a donc $\mathrm J^2 \di x = -\di x$
et $\mathrm J^2\dc x=-\dc x$ pour tout $x\in V$.
Par suite, $\mathrm J^2$ agit comme $(-1)^{p+q}$ sur $\mathrm A^{p,q}(V)$;
en particulier $\mathrm J^2$ est l'identité sur $\mathrm A^{p,p}(V)$.
On prendra garde que notre définition diffère de celle de Lagerberg
qui posait $\mathrm J\dc x=\di x$; elle fait disparaître des termes de signe;

On dira qu'une forme~$\omega$ de type~$(p,q)$
est \emph{symétrique} si l'on a $p=q$ et si $\mathrm J\omega=\omega$.
Les formes symétriques  définissent
une sous-algèbre commutative de~$A^{*,*}(V)$.

\section{Superformes positives}
\label{subsec.formes-positives}
\subsection{}
Comme l'a montré Lagerberg, les superformes fournissent
un analogue réel de l'algèbre des formes différentielles
réelles sur un espace complexe. Cette analogie
se prolonge jusqu'à des notions de positivité
qui jouent un rôle crucial dans la théorie.
Nous commençons ainsi par décrire un certain nombre de cônes
des espaces $A^{p,p}(V)$.
Les résultats qui suivent sont pour l'essentiel des adaptations
d'énoncés classiques dans le cas complexe; certains sont
énoncés (parfois sans preuve) dans~\cite{lagerberg2012}.

\begin{defi}\label{defi.positive-gt-n}
Soit $V$ un espace vectoriel réel de dimension~$n$,
Soit $\omega\in A^{n,n}(V)$.
On dit que $\omega$ est positive
s'il existe une base $(x_1,\dots,x_n)$ de~$V^*$
et un nombre réel $a\geq 0$ tel que
$\omega = a \di x_1\wedge \dc x_1\wedge \dots \wedge \di x_n\wedge\dc{x_n}$.
\end{defi}

C'est alors possible pour toute base de~$V^*$; en effet,
si $f$ est un isomorphisme linéaire de~$V$,
$f^*(\omega)=\det(f)^2 \omega$.

\begin{defi}\label{defi.positive-gt}
Soit $V$ un espace vectoriel réel de dimension~$n$,
soit $p$ un entier naturel et soit $\omega\in A^{p,p}(V)$
une forme \emph{symétrique} de type~$(p,p)$ sur~$V$.
\begin{enumerate}\def\theenumi{\alph{enumi}}\def\labelenumi{\theenumi)}
\item 
On dit que $\omega$ est \emph{faiblement positive}
si pour toute famille $(\alpha_1,\dots,\alpha_{n-p})$ de formes de 
type~$(1,0)$,
la $(n,n)$-forme
\[ \omega \wedge \alpha_1 \wedge \mathrm J(\alpha_1) \wedge\dots
 \wedge \alpha_{n-p} \wedge \mathrm J(\alpha_{n-p})  \]
est positive.
\item 
On dit que $\omega$ est \emph{positive} si pour toute forme~$\alpha$
de type~$(n-p,0)$,
la $(n,n)$-forme
\[ (-1)^{(n-p)(n-p-1)/2} \omega \wedge \alpha \wedge \mathrm J(\alpha) \]
est positive.

\item On dit que $\omega$ est \emph{fortement positive} s'il existe
un entier~$m$, 
des nombres réels positifs~$a_1,\dots,a_m$ 
et des formes~$\alpha_{j,s}$ de type~$(1,0)$ (pour $1\leq j\leq p$
et $1\leq s\leq m$)
tels que
\[ \omega = \sum_{s=1}^m  a_s  \alpha_{1,s}\wedge \mathrm J(\alpha_{1,s})\wedge\dots
\wedge \alpha_{p,s}\wedge \mathrm J(\alpha_{p,s}). \]
\end{enumerate}
\end{defi}

Pour l'explication du signe dans la condition \emph b),
noter que si $\alpha=\alpha_1\wedge\dots\wedge \alpha_m$ est le produit
de $m$ formes de type~$(1,0)$,
alors
\[ 
\alpha_1\wedge\mathrm J(\alpha_1) \wedge\dots \wedge \alpha_m \wedge \mathrm J(\alpha_m)
= (-1)^{m(m-1)/2} \alpha\wedge\mathrm J(\alpha) . \]

Ces cônes, dont la définition est due à Lagerberg~\cite{lagerberg2012},
sont des analogues de cônes classiquement
introduits en géométrie complexe 
(chapitre~III de~\cite{demailly2009}, \S3.2 de~\cite{klimek91}).
Bien qu'inspirée par celle utilisée en géométrie
analytique complexe,
la terminologie adoptée en diffère légèrement: 
les formes faiblement positives sont analogues à
celles qui sont simplement appelées  positives 
en géométrie analytique complexe.

\subsection{}
Une $(0,0)$-forme $a\in A^{0,0}(V)=\R$  fortement positive 
est un nombre réel positif ou nul.

Si $(x_1,\dots,x_n)$ est une base de~$V$,
une $(n,n)$-forme~$\omega$ s'écrit de manière unique
$a \di (x_1)\wedge \dc(x_1)\wedge\dots$;
les trois conditions de positivité sont équivalentes à l'inégalité $a\geq0$.

Nous avons restreint les définitions de positivité aux formes symétriques.
Rappelons que si $\dim(V)=n$, toute $(n,n)$-forme antisymétrique est nulle,
de sorte qu'une $(p,p)$-forme antisymétrique est orthogonale à toute
$(n-p,n-p)$-forme symétrique.


Ces trois conditions définissent
des cônes convexes fermés de l'espace $A^{p,p}_{\mathrm{sym}}(V)$
des formes symétriques de type~$(p,p)$.
Par définition, une forme symétrique
$\alpha$ de type~$(p,p)$ est faiblement positive
si et seulement si $\alpha\wedge\beta$ est positive
pour toute forme~$\beta$ fortement positive de type~$(n-p,n-p)$.
La bidualité des cônes convexes fermés en dimension
finie (theorem~14.1 de~\cite{rockafellar1970})
entraîne que $\alpha$ est fortement positive si et seulement
si $\alpha\wedge\beta$ est  positive pour toute forme faiblement
positive~$\beta$.

Nous verrons plus bas 
(corollaire~\ref{coro.positive-autodual})
que la condition de positivité est auto-duale :
une forme $\alpha$ de type~$(p,p)$
est positive si et seulement si $\alpha\wedge\beta$
est positive pour toute forme positive~$\beta$ de type~$(n-p,n-p)$.

Une forme fortement positive est positive, 
une forme positive est faiblement positive.

\begin{prop}[\cite{lagerberg2012}, proposition 2.2]
\label{prop.fpositive-f*}
Une forme symétrique $\omega$ de type~$(p,p)$ sur~$V$
est faiblement positive si et seulement si 
pour tout espace vectoriel~$W$ de dimension~$p$ et toute application 
linéaire~$f\colon W\ra V$, la forme~$f^*\omega$ est positive.
\end{prop}
\begin{proof}
Soit $\omega$ une forme de type~$(p,p)$ telle que
la forme $f^*\omega$ soit positive
pour tout espace vectoriel~$W$ de dimension~$p$
et toute application linéaire $f\colon W\to V$;
démontrons que $\omega$ est faiblement positive.
Soit $\alpha_1,\dots,\alpha_{n-p}$ des formes de type~$(1,0)$ sur~$V$;
posons 
\[ \mu = \omega \wedge \alpha_1 \wedge \mathrm J(\alpha_1) \wedge
\dots \wedge \alpha_{n-p} \wedge \mathrm J(\alpha_{n-p}) . \]
On doit prouver que $\mu$ est positive.
Si les formes $\alpha_j$  sont linéairement dépendantes, 
\[ \alpha_1 \wedge \mathrm J(\alpha_1) \wedge
\dots \wedge \alpha_{n-p} \wedge \mathrm J(\alpha_{n-p}) = 0, \]
donc $\mu=0$.
On peut donc supposer qu'elles sont linéairement indépendantes.
Soit $\alpha_{n-p+1},\dots,\alpha_n$ des formes de type~$(1,0)$
de sorte que la famille $(\alpha_1,\dots,\alpha_n)$ soit une base de~$V^*$;
soit $(v_1,\dots,v_n)$ la base duale.
Soit $f\colon \R^p\to V$ l'application linéaire 
donnée par $f(x_1,\dots,x_p) = \sum_{i=1}^p x_i v_{n-p+i}$
et soit $g\colon \R^{n-p}\to V$ l'application linéaire
donnée par $g(y_1,\dots,y_{n-p})=\sum_{j=1}^{n-p} y_j v_j$.
Pour tout~$j\in\{1,\dots,n-p\}$,
on a $f^*(\alpha_j) =  0$  et $g^*(\alpha_j)=\di y_j$.
Soit $(f,g)\colon\R^n\to V$ l'application linéaire 
définie par
$(x_1,\dots,x_p,y_1,\dots,y_{n-p})\mapsto f(x_1,\dots,x_p)
+ g(y_1,\dots,y_{n-p})$; c'est un isomorphisme.
De plus,
\begin{align*}
 (f,g)^*\mu & = (f,g)^* (\omega\wedge \alpha_1\wedge \mathrm J(\alpha_1)
  \wedge\dots\wedge \alpha_{n-p}\wedge \mathrm J(\alpha_{n-p}) \\
& = (f,g)^*(\omega) \wedge \di y_1\wedge \dc y_1 \wedge \dots
  \wedge \di y_{n-p}\wedge \dc y_{n-p}. \end{align*}
En outre, $(f,g)^*(\omega)=f^*(\omega) + g^*(\omega)$
et 
\[ g^*(\omega) \wedge \di y_1\wedge \dc y_1 \wedge \dots
  \wedge \di y_{n-p}\wedge \dc y_{n-p}= 0. \]
Ainsi, 
\[  (f,g)^*\mu = f^*(\omega) \wedge \di y_1\wedge \dc y_1 \wedge \dots
  \wedge \di y_{n-p}\wedge \dc y_{n-p}. \]
La positivité de la $(p,p)$-forme~$f^*(\omega)$ entraîne
donc celle de $(f,g)^*(\mu)$. Comme $(f,g)$ est un isomorphisme
linéaire, $\mu$ est positive.

Inversement, soit $W$ un espace vectoriel de dimension~$p$
et soit $f\colon W\to V$ une application linéaire. 
Si $f$ est n'est pas injective, la $(p,p)$-forme $f^*\omega$
est nulle. Supposons donc que $f$ est injective.
Soit $(v_1,\dots,v_n)$ une base de~$V$ tel
que $(v_1,\dots,v_p)$ soit une base de~$f(W)$;
soit $(\alpha_1,\dots,\alpha_n)$ la base duale.
Définissons $\mu$ et $g$ comme précédemment.
Par hypothèse, $(f,g)^*\mu $ est positive; le calcul ci-dessus
entraîne alors que $f^*(\omega)$ est positive.
\end{proof}

\subsection{}
\label{subsec.forme-positive-bil}
Expliquons 
comment la notion de positivité d'une $(p,p)$-forme symétrique~$\omega$ 
peut s'interpréter en termes de positivité 
d'une forme bilinéaire~$q_\omega$ sur~$\Lambda^pV$ canoniquement associée
à~$\omega$.

Pour tout $k\in\N$, posons $\sigma_k=(-1)^{k(k-1)/2}$.
Fixons un élément non nul~$\lambda$ de~$\Lambda^nV^*$.
Par définition les $(n,n)$-formes positives sont les multiples
positifs de $\sigma_n \lambda \wedge \mathrm J(\lambda)$.

On associe à~$\omega$ une application bilinéaire sur~$\Lambda^{n-p}(V^*)$
à valeurs dans $A^{n,n}(V)$ par la formule
$(\alpha,\beta)\mapsto \omega\wedge \alpha\wedge \mathrm J\beta$.

Rappelons aussi que le choix de $\lambda$ fournit un isomorphisme
$v\mapsto \phi_v$ entre~$\Lambda^pV$ et $\Lambda^{n-p}V^*$. 
Pour $v\in\Lambda^pV$, $\phi_v$ est la $(n-p)$-forme 
sur~$V$ donnée par $\phi_v(w)=\langle \lambda ,v\wedge w\rangle$
pour tout $w\in\Lambda^{n-p}V$.
Il en résulte une application bilinéaire $b_\omega$ sur $\Lambda^p(V)$
à valeurs dans~$A^{n,n}(V)$:
pour $v,w\in\Lambda^pV$, on a 
\[ b_\omega(v,w) = \omega\wedge \phi_v \wedge \mathrm J\phi_w. \]

Par ailleurs, puisqu'elle appartient à~$\Lambda^p V^*\otimes\Lambda^pV^*$,
$\omega$ {\og est\fg} une forme bilinéaire~$q_\omega$ sur~$\Lambda^p V$,
symétrique puisque l'on a supposé que $\omega$ est symétrique.
Explicitement, si $\omega =\alpha\otimes \beta$, pour $\alpha,\beta\in\Lambda^pV^*$, on a 
\[ q_\omega (v,w) = \langle \alpha,v\rangle \langle \beta,w\rangle. \]

Pour tous $v,w\in\Lambda^p V$, on a
\begin{equation}
 b_\omega (v,w) = (-1)^{p(n-p)} q_\omega(v,w)  \lambda \wedge\mathrm J\lambda.\end{equation}
Pour démontrer cette formule, on peut se donner
une base $(e_1,\dots,e_n)$ de~$V$ et la base duale $(x_1,\dots,x_n)$
de~$V^*$.  On note $e_I = e_{i_1}\wedge\dots \wedge e_{i_p}$
si $I=\{i_1,\dots,i_p\}$ et $i_1<\dots<i_p$; on définit
$\di x_I$ de manière analogue.
On peut aussi supposer que $\lambda = \di x_1\wedge\dots\wedge \di x_n$
et $\omega=\di x_I \otimes \dc x_J$, où $I,J$ sont des parties à $p$~éléments
de~$\{1,\dots,n\}$.
Si $I$ est une partie à~$p$ éléments de~$\{1,\dots,n\}$,
notons $\eps_I$ la signature de la permutation de~$\{1,\dots,n\}$
qui est croissante et d'image~$I$ sur~$\{1,\dots,p\}$
et est croissante sur~$\{p+1,\dots,n\}$.
Si $K$ et $L$ sont des parties à~$n-p$ éléments de $\{1,\dots,n\}$,
on a 
\[ \omega\wedge \di x_K \wedge \mathrm J\di x_L =
\begin{cases} 
 (-1)^{p(n-p)} \eps_I\eps_J \lambda\wedge\mathrm J\lambda
 & \text{si $K= \complement I$ et $L=\complement J$,} \\
0 &  \text{sinon.}
\end{cases} \]
En outre, $\phi_{e_K}=\eps_K \di x_{\complement K}$,
si bien que 
\[ b_\omega (e_K,e_L) = \begin{cases}
(-1)^{p(n-p)} \lambda \wedge \mathrm J\lambda & \text{si 
$I=K$ et $J=L$, } \\
0 & \text{sinon.} \end{cases} \]
Enfin, 
\[ q_\omega(e_K,e_L) = \langle \di x_I,e_K\rangle \langle \di x_J,e_L\rangle
= \begin{cases} 1 & \text{si $I=K$ et $J=L$, } \\ 0 & \text{sinon,}
\end{cases}
\]
d'où la formule annoncée.

En outre, les formules précédentes
entraînent que les signes de ces diverses formes  sont reliées
par les équivalences suivantes: \\
La forme $\omega$ est positive au sens de la définition~\ref{defi.positive-gt};
\\
$\Leftrightarrow$ L'application bilinéaire $(\alpha,\beta)\mapsto \omega\wedge\alpha\wedge\mathrm J\beta$ est de signe $\sigma_{n-p}$; \\
$\Leftrightarrow$ L'application bilinéaire $b_\omega$ est de signe~$\sigma_{n-p}$; \\
$\Leftrightarrow$ La forme bilinéaire~$q_\omega$ est de signe
$\sigma_{n-p}(-1)^{p(n-p)}\sigma_n = \sigma_p$.

Enfin, si $f\colon W\to V$ est une application linéaire,
on a $q_{f^*\omega}=q_\omega \circ (f,f)$.

Nous dirons qu'une $(p,p)$-forme symétrique est \emph{définie positive}
si la forme bilinéaire symétrique $\sigma_p q_\omega$ est définie positive.

\begin{lemm}\label{lemm.positive-matrice}
Soit $\omega\in A^{p,p}(V)$ une forme \emph{symétrique} de type~$(p,p)$,
soit $(x_1,\dots,x_n)$ une base de~$V$ et soit
\[ \omega=(-1)^{p(p-1)/2} \sum_{I,J} \omega_{I,J} \di x_I \wedge \dc x_J \]
la décomposition de~$\omega$ dans la base correspondante
de~$A^{p,p}(V)$.
Alors, les quatre propriétés suivantes sont équivalentes:
\begin{enumerate}
\item La forme $\omega$ est positive (\resp définie positive) ;
\item La matrice symétrique $(\omega_{I,J})$ (indexée par les parties
à $p$ éléments de~$\{1,\dots,n\}$) est positive (\resp définie positive) ;
\item Il existe une base $(\alpha_I)$ de $\Lambda^pV^*$ 
et une famille $(c_I)$ de nombres réels positifs ou nuls
(\resp strictement positifs), 
indexées par les parties
à $p$ éléments de~$\{1,\dots,n\}$,
telles que
\[ \omega = (-1)^{p(p-1)/2} \sum_{I}  c_I   \alpha_I \wedge \mathrm J(\alpha_I). \]
\end{enumerate}
\end{lemm}
\begin{proof}
L'équivalence des assertions~(1) et~(2) résulte
de la discussion qui précède le lemme et
de ce qu'avec les notations introduites, la matrice~$(\omega_{IJ})$ est celle
de la forme bilinéaire $(-1)^{p(p-1)/2} q_\omega$
dans la base duale de la base $(\di x_I)_I$ de~$\Lambda^pV^*$.
Celle des propriétés~(2) et propriétés~(3) découle de la diagonalisation
des formes quadratiques. 
\end{proof}

\begin{coro}\label{coro.positive-generateurs}
L'ensemble des formes positives  dans $A^{p,p}(V)$ est le cône convexe
fermé engendré par les formes
du type $(-1)^{p(p-1)/2}\alpha\wedge\mathrm J\alpha$,
où $\alpha$ parcourt l'ensemble des formes de type~$(p,0)$.
\end{coro}
\begin{proof}
Le lemme précédent montre que les formes
positives sont contenues dans ce cône. Inversement,
si $\alpha\in A^{p,0}(V)$ est une forme de type~$(p,0)$
et $\beta$ est une $(n-p,0)$-forme, 
la formule
\[ (-1)^{p(p-1)/2} \alpha\wedge\mathrm J\alpha \wedge 
   (-1)^{(n-p)(n-p-1)/2} \beta\wedge\mathrm J\beta
 =  (-1)^{n(n-1)/2}   (\alpha\wedge\beta) \wedge \mathrm J(\alpha\wedge\beta)\]
montre que $(-1)^{p(p-1)/2}\alpha\wedge\mathrm J\alpha$
est une forme positive, d'où l'autre inclusion.
\end{proof}

\begin{coro}\label{coro.positive-autodual}
La condition de positivité est auto-duale:
une forme $\alpha$, symétrique de type~$(p,p)$,
est positive si et seulement si $\alpha\wedge\beta$
est positive pour toute forme positive $\beta$ de type~$(n-p,n-p)$.
\end{coro}
\begin{proof}
Cela résulte immédiatement du corollaire précédent.
\end{proof}

\begin{coro}
Les trois notions de positivité coïncident 
pour les formes de bidegrés $(0,0)$, $(1,1)$, $(n-1,n-1)$ et $(n,n)$.
\end{coro}
Elles ne coïncident pas en général, pas plus que leurs analogues
en géométrie complexe.
\begin{proof}
Le cas des formes de type~$(0,0)$ et $(n,n)$ a déjà été expliqué.
Par définition même, les formes faiblement positives de type~$(n-1,n-1)$
sont positives. Le lemme précédent, ou un argument de dualité,
prouve que les formes positives de type~$(1,1)$ sont fortement positives.

Soit $\omega$ une forme de type~$(1,1)$
qui est faiblement positive.
Dans une base $(x_1,\dots,x_n)$ de~$V$,
écrivons $\omega=\sum_{i,j=1}^n \omega_{i,j} \di x_i \wedge \dc x_j$.
D'après la proposition~\ref{prop.fpositive-f*} appliquée
à l'application linéaire $f\colon \R\to V$ donnée par $t\mapsto tv$,
on  a \[ \sum_{i,j=1}^n \omega_{i,j} x_i(v) x_j(v) \geq 0 \]
pour tout vecteur~$v$ de~$V$, de composantes $v_i=x_i(v)$.
La matrice $(\omega_{i,j})$ est donc positive;
d'après le lemme~\ref{lemm.positive-matrice}, la forme~$\omega$
est positive, et aussi fortement positive.

Par dualité, les formes de type~$(n-1,n-1)$
faiblement positives sont aussi fortement positives.
\end{proof}

\begin{lemm}\label{lemme.image.inverse}
Soit $\omega$ une forme symétrique de type~$(p,p)$ sur~$V$.
Soit $W$ un espace vectoriel réel
et $f\colon W\ra V$ une application linéaire.

\begin{enumerate}
\item
Si $\omega$ est  positive (\resp faiblement positive, \resp fortement positive),
il en est de même de la forme $f^*\omega$ sur~$W$.

\item
Si~$f$ est surjective et si~$f^*\omega$ est  positive
(\resp faiblement positive, \resp fortement positive),
il en est de même de~$\omega$.
\end{enumerate}
\end{lemm}

 
\begin{proof} Commençons par l'assertion~1). 
Le cas des formes fortement positives est évident,
celui des formes faiblement positives résulte 
de la proposition~\ref{prop.fpositive-f*},
tandis que  celui des formes positives découle de l'interprétation
de la positivité en termes de formes bilinéaires
(\S\ref{subsec.forme-positive-bil}).

Venons-en à l'assertion~2) ; on suppose
donc~$f$ surjective. 
Le cas des formes faiblement positives
résulte de la proposition~\ref{prop.fpositive-f*}
et du fait que toute application affine~$g\colon U\ra V$ 
se relève à~$W$. 

En ce qui concerne les formes positives
et fortement positives,
on choisit un isomorphisme~$W\simeq V\oplus V'$
modulo lequel~$f$ est la première projection. Si~$f^*\omega$
est positive, la forme bilinéaire~$\eta$ 
qu'elle induit sur~$\Lambda^p W$ est positive 
(\S\ref{subsec.forme-positive-bil}) ; comme~$\eta$
est le prolongement
par~$0$ de la forme bilinéaire~$\eta'$
induite par~$\omega$
sur le sommande $\Lambda^p V$ de~$\Lambda^pW$, 
la forme~$\eta$' est positive, et~$\omega$ est positive en vertu 
là encore de~\S\ref{subsec.forme-positive-bil}.

Si~$f^*\omega$ est fortement positive, elle s'écrit
sous la forme
\[\sum a_i \alpha_{i,1}\wedge \mathrm J(\alpha_{i,1})\wedge \alpha_{i,2}\wedge \mathrm J(\alpha_{j,2})\wedge\ldots
\wedge \alpha_{i,n_i}\wedge \mathrm J(\alpha_{i,r_i}),\]
où
les~$\alpha_{i,j}$
sont de type~$(1,0)$ et les~$a_i$ strictement positifs. 
Comme~$f^*\omega$ est 
égale au prolongement par~$0$ de la forme~$\omega$
vue comme vivant dans le sommande~$A^{p,p}(V)$ de~$A^{p,p}(W)$, 
il vient
\[\omega=\sum a_i \alpha_{i,1,V}\wedge \mathrm J(\alpha_{i,1,V})\wedge \alpha_{i,2,V}\wedge \mathrm J(\alpha_{j,2,V})\wedge\ldots
\wedge \alpha_{i,n_i,V}\wedge \mathrm J(\alpha_{i,r_i,V}),\]
où le~$V$ en indice indique la projection sur le sommande~$A^{1,0}(V)$. Ainsi,
$\omega$ est fortement positive. 
\end{proof}

\begin{prop}[cf. \cite{lagerberg2012}, Lemma 2.1]
\label{prop.produit-positives}
Soit $\omega_1,\dots,\omega_s$ des formes sur~$V$
et posons $\omega=\omega_1\wedge\dots\wedge\omega_s$.

\begin{enumerate}\def\theenumi{\alph{enumi}}\def\labelenumi{\theenumi)}
\item
Si les $\omega_j$ sont toutes fortement positives,
alors $\omega$ est fortement positive.

\item
Si les $\omega_j$ sont toutes positives, alors $\omega$ est positive.

\item 
Si les $\omega_j$ sont toutes faiblement positives 
et si au plus une d'entre elle n'est pas fortement positive,
alors $\omega$ est faiblement positive.
\end{enumerate}
\end{prop}
\begin{proof}
Par récurrence, l'assertion~\emph b) découle de la caractérisation 
du cône des formes positives et de la formule
\[ (-1)^{p(p-1)/2} \alpha\wedge\mathrm J\alpha \wedge 
   (-1)^{q(q-1)/2} \beta\wedge\mathrm J\beta
 =  (-1)^{r(r-1)/2}   (\alpha\wedge\beta) \wedge \mathrm J(\alpha\wedge\beta),\]
où $\alpha$, $\beta$ sont des formes de type~$(p,0)$ et~$(q,0)$
respectivement, et $r=p+q$.

Pour démontrer les assertions~\emph a) et~\emph c),
il suffit par récurrence
de prouver que si une forme $\beta$ est fortement positive,
alors $\alpha\wedge\beta$ est de même nature que $\alpha$.
On peut supposer que $\beta$ est de la forme 
$\beta_1\wedge \mathrm J\beta_1\wedge\dots\wedge\beta_m\wedge\mathrm J\beta_m$,
où les $\beta_j$ sont des formes de type~$(1,0)$.
La preuve est alors directe.
\end{proof}

\begin{prop}[\cite{harvey1977}, corollaire~1.17]
\label{prop.harvey}
Soit  $\omega$ une $(1,1)$-forme symétrique définie positive sur~$V$.
Soit $p$ un entier~$\geq 1$.
Dans l'espace des $(p,p)$-formes symétriques, 
la forme~$\omega^p$ appartient à l'intérieur relatif
du cône des $(p,p)$-formes   fortement positives,
à l'intérieur du cône des $(p,p)$-formes symétriques positives,
et à l'intérieur du cône des~$(p,p)$-formes faiblement positives. 
\end{prop}

En particulier, les formes
positives, et a fortiori les formes faiblement positives,
engendrent l'espace des formes symétriques.
En revanche, Lagerberg donne dans \cite{lagerberg2012} (remark~4.1)
un exemple de $(2,2)$-forme 
faiblement positive sur~$\R^4$, non nulle, qui est orthogonale 
à toute forme fortement positive; 
ainsi les formes fortement positives engendrent
un sous-espace strict des $(2,2)$-formes  symétriques sur~$\R^4$.

  
\begin{proof}
D'après la proposition~\ref{prop.produit-positives}, 
$\omega^p$ est fortement positif.

Nous allons démontrer que 
si $\alpha\in A^{n-p,n-p}(V)$ est une  forme faiblement positive et symétrique
telle que $\omega^p\wedge\alpha= 0$,
alors $\beta\wedge \alpha=0$
pour toute forme~$\beta$ fortement positive.
Cela entraîne la première partie de la proposition, par dualité.
Cela entraîne aussi la troisième partie. Soit en effet
une forme~$\alpha$, fortement positive, orthogonale à~$\omega^p$;
démontrons qu'elle est nulle.
Comme elle est également faiblement positive, elle est alors
orthogonale à toute forme fortement positive.
Puisque $\alpha$ est fortement positive, elle est 
somme de formes du type
$ c \alpha_1\wedge\mathrm J\alpha_1 \wedge \dots \wedge
\alpha_{n-p}\wedge\mathrm J\alpha_{n-p}$, où $(\alpha_1,\dots,\alpha_{n-p})$
est une famille d'éléments de~$V^*$, qu'on peut supposer libre,
et $c\in\R_+$.
Par positivité, chacun de ces
termes est également orthogonal à toute fortement positive.
Mais cela entraîne qu'ils sont nuls: complétons en effet la famille $(\alpha_1,\dots,\alpha_{n-p})$ en une base~$(\alpha_1,\dots,\alpha_n)$;
en prenant $\beta=
 \alpha_{n-p+1}\wedge\mathrm J\alpha_{n-p+1} \wedge \dots $,
on trouve 
$ c \alpha_1\wedge\mathrm J\alpha_1 \wedge \dots \wedge
\alpha_{n}\wedge\mathrm J\alpha_{n}\geq 0$, d'où $c=0$.

%
Démontrons maintenant l'énoncé admis. Soit $\alpha$ une forme symétrique
faiblement positive de type~$(n-p,n-p)$ telle que $\omega^p\wedge\alpha=0$,
soit $\beta_1,\dots,\beta_p$ dans~$V^*$ et soit
$\beta=\beta_1\wedge\mathrm J\beta_1 \wedge \dots \wedge
\beta_p\wedge\mathrm J\beta_p$;  prouvons que $\beta\wedge\alpha=0$.

C'est évident si $\beta_1,\dots,\beta_p$ ne sont pas linéairement
indépendantes, puisqu'alors $\beta=0$. Supposons donc
que $\beta_1,\dots,\beta_p$ soient linéairement indépendantes
et prolongeons cette famille en une base $(\beta_1,\dots,\beta_n)$
de~$V^*$; posons $\eta=\sum_{j=1}^n \beta_j\wedge\mathrm J\beta_j$.

Comme $\omega$ est symétrique définie positive, 
il existe un nombre réel strictement positif~$c$
tel que $\omega \geq c\eta$ 
au sens où la différence $\omega-c\eta$ 
est une $(1,1)$-forme symétrique fortement positive.
D'après la proposition~\ref{prop.produit-positives},
appliquée à la forme faiblement positive~$\alpha$
et aux formes fortement positives $\omega$, $\omega-c\eta$ et $\eta$,
la différence $\omega^p \wedge\alpha-  c^p \eta^p\wedge\alpha$
est une $(n,n)$-forme faiblement positive,
de même que la forme $\eta^p\wedge\alpha$.
Comme $\omega^p\wedge\alpha=0$ par hypothèse, on en déduit
que $\eta^p\wedge\alpha=0$.

Pour toute partie~$I=\{i_1,\dots,i_p\}$ de~$\{1,\dots,n\}$, où
$i_1<\dots<i_p$, notons $\beta_I=\beta_{i_1}\wedge\dots\wedge\beta_{i_p}$;
notons aussi $\lambda=\beta_1\wedge\dots\wedge\beta_n$.
Si $p\in\N$, posons enfin $\sigma_p=(-1)^{p(p-1)/2}$.
Alors, $\eta^p= p!\sigma_p \sum_{\abs I=p} \beta_I \wedge \mathrm J\beta_I$.
La famille $(\beta_J\wedge\mathrm J\beta_K)$ indexée par les couples
$(J,K)$ de parties à $n-p$ éléments de~$\{1,\dots,n\}$
est une base de $A^{n-p,n-p}(V)$.
Écrivons donc~$\alpha$ dans cette base: 
\[ \alpha = \sigma_{n-p} \sum_{\abs J=\abs K=n-p} \alpha_{J,K} \beta_J \wedge \mathrm J\beta_K.   \]
Pour toute partie $I$ à $p$ éléments de~$\{1,\dots,n\}$, on a donc
\begin{align*}
\sigma_{n-p} \beta_I\wedge\mathrm J\beta_I \wedge\alpha & 
= \sigma_p \sigma_{n-p}
\sum_{J,K}  \alpha_{J,K} \beta_I \wedge \mathrm J\beta_I \wedge\beta_J \wedge\mathrm J\beta_ K  
= \sigma_n  \alpha_{\complement I,\complement I} \lambda\wedge\mathrm J\lambda, \\
\noalign{\noindent{d'où}}
\eta^p \wedge \alpha 
& = p!\sigma_n  \sum_{\abs J=n-p}  \alpha_{J,J}  \lambda\wedge\mathrm J\lambda .
\end{align*}

Comme $\alpha$ est faiblement positif et $\sigma_{n-p}\beta_I\wedge\mathrm J\beta_I$ est fortement positive pour toute partie~$I$,
la proposition~\ref{prop.produit-positives}
entraîne aussi que $\alpha_{J,J}\geq 0$ pour toute partie~$J$. 
Puisque $\eta^p\wedge\alpha=0$, il s'ensuit que $\alpha_{J,J}=0$ pour tout~$J$.
C'est en particulier le cas pour $J=\{p+1,\dots,n\}$;
comme $\beta = \pm \beta_I\wedge \mathrm J\beta_I$ pour $I=\{1,\dots,p\}$, 
on a donc 
$\beta\wedge \alpha=0$, ce qu'il fallait démontrer.

Pour établir la seconde partie de la proposition, on démontre
que si une forme positive $\alpha$ vérifie $\omega^p\wedge\alpha=0$,
alors $\beta\wedge\alpha=0$ pour toute forme positive~$\beta$
de type~$(p,p)$.

Considérée comme forme bilinéaire sur $A^{p,0}$, la forme
$\omega^p$ est symétrique définie positive,
il existe un nombre réel~$\eta>0$ tel que $\omega^p-\eta\beta$
est positive.
Alors, $0=\omega^p\wedge\alpha $
est la somme des formes 
$(\omega^p-\eta\beta)\wedge\alpha$ et $\eta \beta\wedge\alpha$,
et ces formes sont positives d'après la proposition~\ref{prop.produit-positives}.
Par suite, elles sont toutes deux nulles. En particulier, $\beta\wedge\alpha=0$.
\end{proof}

On dira qu'une forme symétrique~$\omega$ de type~$(p,p)$
est strictement  positive (\resp strictement faiblement positive, 
\resp strictement fortement positive)
s'il existe une $(1,1)$-forme définie  positive~$\lambda$
telle que $\omega-\lambda^p$ soit une forme positive
(\resp faiblement, \resp fortement positive).

\section{Superformes différentielles, d'après A.~Lagerberg}

\subsection{}
Soit $V$ un espace affine de dimension finie sur~$\R$;
on note $\vec V$ son espace directeur et $\vec V^*$ le dual de~$\vec V$.

Dans~\cite{lagerberg2012}, Lagerberg
définit comme suit
la notion de superforme différentielle
de type~$(p,q)$\index{superforme} sur un 
ouvert~$U$
de~$V$, que nous appellerons simplement dans ce qui suit 
\emph{forme de type $(p,q)$}\index{forme de type $(p,q)$! sur un espace affine} : 
c'est un élément de 
\[ \mathscr A_V^{p,q}(U) 
= \mathscr C^\infty(U) \otimes_\R   A^{p,q}(\vec V)
=  \mathscr C^\infty(U) 
       \otimes_\R \Lambda^p \vec V^*\otimes_\R \Lambda^q\vec  V^*. \]
Elles sont bien sûr nulles si $p>\dim(V)$ ou $q>\dim(V)$.
En particulier, $\mathscr A_V^{0,0}(U)$ est l'anneau
des fonctions $\mathscr C^\infty$ sur $U$;
nous le noterons plus simplement $\mathscr A_V(U)$. 

On considérera parfois une superforme de type~$(p,q)$ sur un
ouvert~$U$ de~$V$
comme une application~$\mathscr C^\infty$ de~$U$ 
vers l'espace vectoriel $A^{p,q}(\vec V)$.

\subsection{}
Une fois choisie une base de~$\vec V$ et une origine de~$V$,
on dispose de coordonnées
$(x_1,\dots,x_n)$ qui permettent d'identifier~$V$ à~$\R^n$
(où $n=\dim(V)$).
Pour distinguer les formes linéaires apparaissant
dans le premier facteur~$\Lambda^p \vec V^*$ de celles apparaissant
dans le second~$\Lambda^q \vec V^*$, on convient de noter
systématiquement les premières avec le préfixe~$\di$
et les secondes avec le préfixe~$\dc$.
Une forme de type~$(p,q)$ sur~$U$ 
s'écrit ainsi
\begin{align*}
 \omega  & = \sum_{\abs I=p,\, \abs J=q} \omega_{I,J} \di 
x_I\otimes \dc x_J \\
& = \sum_{\abs I=p, \,\abs J=q} \omega_{I,J} \, \di 
x_{i_1}\wedge\dots \wedge \di x_{i_p} \otimes \dc x_{j_1}\wedge\dots
\wedge \dc x_{j_q}, \end{align*}
où les $\omega_{I,J}$ sont des fonctions~$\mathscr C^\infty$ sur~$U$,
et où~$I=(i_1,\dots,i_p)$ et $J=(j_1,\dots,j_q)$ 
parcourent respectivement
les suites strictement croissantes d'entiers compris entre~$1$ et~$n$.

\subsection{}
La somme directe $\mathscr A_V^{*,*}(U)=\bigoplus \mathscr A_V^{p,q}(U)$ 
est naturellement munie d'une structure d'algèbre 
graduée 
sur l'anneau $\mathscr A_V(U)$
obéissant à la règle de commutativité
suivante: si $\omega$ est de type~$(p,q)$ et $\omega'$ de 
type~$(p',q')$,
alors
\begin{equation}
 \omega' \wedge \omega = (-1)^{(p+q)(p'+q')} \omega\wedge\omega'. \end{equation}
Ainsi, la forme $\di x_I \otimes \dc x_J$
sera plutôt notée $\di x_I \wedge \dc x_J$.

Cette algèbre est munie d'une involution~$\mathrm J$,
définie, en coordonnées, par la formule
\begin{equation}
\label{eq.J}
 \mathrm J\omega = (-1)^q \sum_{I,J} \omega_{I,J}\,\dc x_I\wedge \di x_J
= (-1)^{(p+1)q} \sum_{I,J} \omega_{I,J}\,\di x_J \wedge 
\dc x_I.
\end{equation}
Par suite, $\mathrm J^2$ agit comme $(-1)^{p+q}$ sur $\mathscr A^{p,q}_V$.

Lorsque $p=q$, on dit qu'une forme~$\omega$ de type~$(p,p)$ 
est \emph{symétrique} si elle vérifie $\mathrm J\omega=\omega$;
cela revient donc à exiger que $\omega_{I,J}=\omega_{J,I}$ 
pour tout couple~$(I,J)$.

%
%

\subsection{}
Soit $f\colon V'\ra V$ une application affine entre 
deux espaces
affines de dimension finie.
On définit de façon évidente l'image réciproque
$f^*\omega$ d'une forme~$\omega$ de type $(p,q)$ 
sur un ouvert~$U$ 
de~$V$;
c'est une forme de type $(p,q)$ 
sur $f^{-1}(U)$.

Ces applications~$f^*$ sont compatibles aux structures
d'algèbres, ainsi qu'aux involutions~$\mathrm J$.

\subsection{Calcul différentiel}
Les faisceaux $\mathscr A_V^{p,q}$ sont munis 
de deux opérateurs différentiels, $\di$ et $\dc$,
définis ainsi en coordonnées. Pour $\omega$ comme ci-dessus,
\begin{align}
\label{eq.di}
\di\omega = \sum_{\abs I=p, \,\abs J=q} \sum_{i=1}^n 
 \frac{\partial\omega_{I,J}(x)}{\partial x_i} \di x_i \wedge \di x_I 
\wedge \dc x _J  \\
\label{eq.dc}
\dc  \omega =(-1)^p \sum_{\abs I=p, \,\abs J=q} \sum_{j=1}^n 
 \frac{\partial\omega_{I,J}(x)}{\partial x_j} \di x_I \wedge \dc x_j 
\wedge  \dc x_J  .
\end{align}

Ainsi, $\di\omega$ est de type $(p+1,q)$, tandis que $\dc\omega$
est de type $(p,q+1)$. Notons que pour $i\in\{1,\dots,n\}$,
les symboles $\di x_i$ et $\dc x_i$ sont bien les images
par~$\di$ et~$\dc$ de la fonction affine~$x_i$ sur~$V$.
Remarquons aussi que la différentielle~$\di$ 
n'est autre que l'opérateur induit par la connexion  naturelle dans le
fibré vectoriel trivial de fibre~$\Lambda^q\vec V^*$.

Le lemme de Schwarz entraîne comme de coutume
les relations $\di\di\omega=\dc\dc\omega=0$
et $\di\dc\omega=-\dc\di\omega$.

\begin{lemm}\label{lemm.d'-d''-J}
On a les relations 
\begin{equation}
\mathrm J\di=\dc \mathrm J \quad\text{et}\quad \mathrm J\dc=-\di \mathrm J . \end{equation}
Par suite, 
sur $\mathscr A^{p,q}_V$, on a les relations
\begin{equation}
\dc = \mathrm J\di \mathrm J^{-1} 
  = - \mathrm J^{-1} \di \mathrm J
  = (-1)^{p+q}\mathrm J \di \mathrm J
\end{equation}
et
\begin{equation}
\di=\mathrm J^{-1}\dc \mathrm J
 =  - \mathrm J \dc \mathrm J^{-1}
=(-1)^{p+q}\mathrm J\dc\mathrm J.\end{equation}
On a aussi 
\begin{equation} \ddc\mathrm J = \mathrm J\ddc. \end{equation}
\end{lemm}
\begin{proof}
Par linéarité, il suffit de traiter
le cas d'une forme du type $\omega=\alpha\di x_I\wedge \dc x_J$,
où $I$ et $J$ sont des multi-indices de longueur~$p$ et~$q$.
D'après la formule~\eqref{eq.J}, on a 
\begin{align*}
\mathrm J\di\mathrm J\omega & 
= \mathrm J \di \mathrm J \left( \alpha \di x_I\wedge\dc x_J\right)\\
& = (-1)^{(p+1)q} \mathrm J\di \left( \alpha \di x_J\wedge\dc x_I\right) \\
& = (-1)^{(p+1)q} \mathrm J \sum_{j=1}^n \frac{\partial \alpha}{\partial x_j}
\di x_j \wedge \di x_J\wedge \dc x_I \\
&= (-1)^{(p+1)q} (-1)^{(q+2)p}
  \sum_{j=1}^n  \frac{\partial \alpha}{\partial x_j}
      \di x_I\wedge \dc x_j \wedge\dc x_J \\
& = (-1)^{q} \sum_{j=1}^n  \frac{\partial \alpha}{\partial x_j}
\di x_I \wedge \dc x_i\wedge \dc x_J  \\
& = (-1)^{p+q} \dc\omega,
\end{align*}
d'après la formule~\eqref{eq.dc}.

Puisque $\dc\omega$ est de type~$(p,q+1)$,
cela démontre l'égalité $\mathrm J\di\mathrm J=-\mathrm J^2\dc$,
autrement dit $\di\mathrm J=-\mathrm J \dc$.
%
%
%
Les autres relations s'en déduisent, compte tenu du fait que $\mathrm J^2$
vaut $(-1)^{p+q}$ sur~$\mathscr A^{p,q}_V$.
\end{proof}

\subsection{}
Les opérateurs $\di$ et $\dc$ sont des dérivations de l'algèbre
graduée $\mathscr A^{*,*}(U)$.
Si $\omega$ et~$\omega'$ sont respectivement de type~$(p,q)$
et~$(p',q')$, on a
\begin{gather*}
 \di (\omega\wedge \omega')
=\di\omega\wedge\omega'+(-1)^{p+q}\omega\wedge\di\omega', \\
 \dc (\omega\wedge \omega')
=\dc\omega\wedge\omega'+(-1)^{p+q}\omega\wedge\dc\omega'.
\end{gather*}

\begin{lemm}\label{lemm.didc=-dcdi}
Soit $p,q,r$ des entiers naturels tels que $p+q+r=\dim(V)-1$.
Soit $\alpha, \beta, \gamma$ des superformes
\emph{symétriques} de types $(p,p)$, $(q,q)$ et $(r,r)$, respectivement.
On suppose que $\alpha$ et $\beta$ sont de classe~$\mathscr C^1$.
On a l'égalité
$\di\alpha\wedge\dc\beta \wedge\gamma=-\dc\alpha\wedge\di\beta \wedge\gamma$.
\end{lemm}
\begin{proof}
Puisque $\alpha$ est de type $(p,p)$ et symétrique, on a 
\[ \mathrm J \di\alpha = \dc \mathrm J \alpha=\dc\alpha.\]
De même, 
$ \mathrm J\dc\beta = - \di \mathrm J \beta$.
Alors,
\[ \mathrm J (\di\alpha\wedge \dc\beta) 
    = \mathrm J \di \alpha \wedge \mathrm J\dc\beta  
    =  - \dc  \alpha \wedge \di \beta. \]
Puisque $\gamma$ est symétrique, on a donc
\[ \mathrm J\omega = \mathrm J (\di\alpha\wedge \dc\beta\wedge\gamma) 
    =  - \dc  \alpha \wedge \di \beta\wedge\gamma = -\eta. \]
La forme 
$ \di\alpha\wedge \dc\beta\wedge\gamma $
est de type $(p+1,p)+ (q,q+1)+(r,r)=(p+q+r+1,p+q+r+1)$ ;
comme $p+q+r+1=\dim(V)$, elle est symétrique. 
On a donc la relation annoncée.
\end{proof}

\begin{rema}
Comme on le voit, le dédoublement des variables différentielles
introduit par Lagerberg dans~\cite{lagerberg2012}
permet un calcul différentiel avec des formes de type~$(p,q)$
(pour $0\leq p,q\leq n$)
sur un espace réel de dimension~$n$
qui est formellement analogue au calcul différentiel holomorphe.
Expliquons comment on peut l'interpréter en termes de géométrie
tropicale complexe. Soit $T$ le tore complexe $(\C^*)^n$,
soit $V=\R^n$ et soit $\pi\colon T\ra V$ l'application
de tropicalisation donnée par $(z_1,\dots,z_n)
\mapsto (\log\abs{z_1},\dots,\log\abs{z_n})$.
Elle identifie~$V$ au quotient de~$T$ par son sous-groupe
compact maximal~$\mathbf U$. 
Interprétant $\di x_j$ comme $\mathrm d\log\abs{z_j}$
et $\dc x_j$ comme $\mathrm d\Arg(z_j)$, toute
superforme de type~$(p,q)$ sur un ouvert~$U$ de~$V$
définit une forme différentielle de degré~$p+q$ sur~$\pi^{-1}(U)$
qui est invariante par l'action de~$\mathbf U$.
Inversement, toute forme~$\omega$ de degré~$p+q$ sur~$\pi^{-1}(U)$
peut s'écrire sous la forme
\[ \sum_{I,J} f_{I,J}(r,\theta) \mathrm dr_I \wedge\mathrm d\theta_J,\]
où $(r_1,\theta_1),\dots,(r_n,\theta_n)$ sont respectivement module
et argument de $z_1,\dots,z_n$
(de sorte que $(\log r_1,\dots,\log r_n)\in U$
et $(e^{i\theta_1},\dots,e^{i\theta_n})\in\mathbf U$).
Si, de plus, $\omega$ est invariante sous l'action de~$\mathbf U$,
on obtient que $f_{I,J}$ ne dépend pas de~$\theta$,
si bien que $\omega$ provient d'une superforme de type~$(p,q)$
sur~$\R^n$. Autrement dit, on a un isomorphisme
$\mathscr A^{p,q}_{\text{Lag}}
\simeq (\pi_*) \big((\mathscr A^{p+q}_{T})^{\mathbf U}\big)$
de faisceaux sur~$V$.

Par contre, cette identification ne préserve pas le type
des formes en général ; ce n'est déjà pas
le cas pour $\mathrm d\log\abs{z_j}$ qui n'est
pas de type~$(1,0)$. Toutefois, les
formes ainsi obtenues sur~$T$ sont réelles,
et 
les superformes  \emph{symétriques} de type~$(p,p)$
donnent lieu à  des formes de type~$(p,p)$.
\end{rema}

\section{Superformes à coefficients}

\subsection{}
Soit $V$ un espace affine 
et soit~$\mathscr F$ un sous-faisceau en $\R$-espaces  vectoriels 
du faisceau des fonctions numériques sur~$V$.

Pour tout ouvert~$U$ de~$V$,
une \emph{forme 
de type $(p,q)$  à coefficients dans~$\mathscr F$ sur l'ouvert~$U$} est un 
élément de
 \[ \mathscr F_V^{p,q}(U)
:=\mathscr F(U)\otimes_\R A^{p,q}(\vec V). \]
Cela définit un faisceau en $\R$-espaces vectoriels.

La structure de $\R$-espace vectoriel gradué sur~$A^{*,*}(\vec V)$
et son opérateur~$\mathrm J$
induisent sur la somme directe~$\bigoplus_{p,q}\mathscr F_V^{p,q}$
une structure de faisceau en $\R$-espaces vectoriels gradués
et le munissent d'un opérateur~$\mathrm J$.

Si $\mathscr F$ est un $\mathscr A_V$-module,
alors on peut écrire
\[ \mathscr F_V^{p,q} (U)
 =\mathscr F_V(U)\otimes_{\mathscr A_V(U)}\mathscr A_V^{p,q}(U) , \]
de sorte que $\mathscr F_V^{*,*}$ est un module gradué
sur l'algèbre graduée $\mathscr A_V^{*,*}$;
l'opérateur~$\mathrm J$ est un morphisme d'algèbres.

Si $\mathscr F$ est une $\R$-algèbre,
la structure d'algèbre graduée sur~$\mathscr A^{*,*}(\vec V)$
induit sur~$\mathscr F_V^{*,*}$ une structure de $\R$-algèbre graduée.

\subsection{} \label{valeur-ponctuelle}
Soit~$V$ un espace affine de dimension finie  et soit~$\omega$
une forme de type~$(p,q)$ 
définie sur un ouvert~$U$ de~$V$. Pour tout~$x\in U$, 
on peut évaluer~$\omega$ en~$x$, et obtenir
ainsi un élément~$\omega(x)$ de~$A^{p,q}(\vec V)$. 
En coordonnées, si~$\omega =\sum \omega_{I,J} \di x_I\wedge \dc x_J$
alors~$\omega(x) =\sum \omega_{I,J} (x)\di x_I\wedge \dc x_J$. 
On dira que $\omega$ est ponctuellement nulle en~$x$ si~$\omega(x)=0$.

Observons que l'on a $\omega=0$ si et seulement si~$\omega(x)=0$ 
pour tout~$x\in U$. 

L'évaluation est compatible à la structure de~$\mathscr A_V$-algèbre.
 

\subsection{Calcul différentiel}
Les opérateurs différentiels~$\di$ et~$\dc$
s'étendent de façon évidente aux formes
à coefficients dans~$\mathscr F$
pourvu que les éléments de~$\mathscr F$ aient une régularité suffisante.
Par exemple, si $\omega$ est une forme de type~$(p,q)$
à coefficients de classe~$\mathscr C^r$
pour~$r\geq 1$, alors
$\di\omega$ et $\dc\omega$ sont des formes
de types~$(p+1,q)$ et $(p,q+1)$ à coefficients de classe~$\mathscr C^{r-1}$.
De même, si $\omega$ est à coefficients de classe~$\mathscr C^2$,
le lemme de Schwarz entraîne
encore $\di\di\omega=\dc\dc\omega=0$.

\subsection{}
Soit $f\colon V'\ra V$ une application affine entre 
deux espaces affines de dimension finie.
Soit $\mathscr F'$ un faisceau en $\R$-espaces vectoriels 
de fonctions numériques sur~$V'$; supposons que $u\circ f$
appartienne à~$\mathscr F'$ pour toute fonction~$u$ de~$\mathscr F$.
On définit alors de façon évidente l'image réciproque
$f^*\omega$ d'une forme~$\omega$ de type $(p,q)$ 
 à coefficients dans $\mathscr F$
sur un ouvert~$U$ 
de~$V$;
c'est une forme de type $(p,q)$ 
 à coefficients dans~$\mathscr F'$ 
sur $f^{-1}(U)$.

Ces applications~$f^*$ sont compatibles aux opérateurs~$\mathrm J$,
et aux structures de modules ou d'algèbres quand elles sont définies.

\subsection{}
Si $\mathscr G$ est un sous-faisceau en $\R$-espaces vectoriels de~$\mathscr F$,
alors 
$\mathscr G_V^{*,*}$ est un sous-faisceau
en $\R$-espaces vectoriels gradués
de~$ \mathscr G_V^{*,*}$.

Cela permet par exemple
de voir les formes de classe~$\mathscr C^k$ comme 
des formes à coefficients boréliens.

\section{Intégrale d'une $(p,n)$-forme 
sur un sous-espace de dimension~$p$}

\subsection{}
Soit $\omega$ une $(n,n)$-forme à coefficients boréliens
sur un ouvert~$U$ de~$\R^n$, 
donnée en coordonnées par la formule
\begin{equation}
\label{eq.omega.n-n-d'd''}
  \omega = \omega^\sharp(x_1,\dots,x_n) \di x_1\wedge\dc x_1
\wedge\dots\di x_n\wedge \dc x_n, \end{equation}
où $\omega^\sharp$ est donc une fonction borélienne sur~$U$.
Si $\omega^\sharp$ est intégrable sur~$U$, 
Lagerberg a défini l'intégrale  de~$\omega$  
par la formule
\begin{equation}\label{eq.int.lagerberg}
 \int_U \omega = \int_{U}  \omega^\sharp(x_1,\dots,x_n) 
\,
\mathrm d x_1\dots\mathrm d x_n, \end{equation}
où  $\mathrm dx_1\dots\mathrm d x_n$ désigne la mesure de Lebesgue.\footnote
{En fait, la définition~2.2 de~\cite{lagerberg2012} propose
une autre convention, erronée, mais qu'il n'utilise apparemment pas dans 
son article. La convention que nous adoptons a pour vertu
que l'intégrale d'une $(n,n)$-forme positive 
(définition~\ref{defi.positive-gt})
est positive.}

\subsection{Non-canonicité des intégrales de Lagerberg}
Si $f\colon \R^n\ra\R^n$ est une application affine bijective, 
d'application linéaire associée~$\vec f$,
on a 
\begin{equation}
 \int_{f^{-1}(U)} f^*\omega = \abs{\det(\vec f)} \int_U \omega . \end{equation}

Soit $n$ un entier naturel 
et soit $V$ un espace affine de dimension~$n$ sur~$\R$.
Pour définir l'intégrale d'une $(n,n)$ forme 
intégrable~$\omega$ sur un ouvert~$U$ de~$V$,
il suffit de fixer un isomorphisme affine de~$V$ sur~$\R^n$
et d'appliquer la définition ci-dessus. Le problème est que
\emph{le résultat obtenu dépend de l'isomorphisme choisi} : 
en effet, par ce qui précède, deux isomorphismes~$u$
et~$u'$ fournissent la même notion d'intégrale si et seulement 
le jacobien de l'automorphisme affine~$u'\circ u^{-1}$ de~$\R^n$
est égal à~$\pm 1$.

Pour définir l'intégrale d'une $(n,n)$-forme
sur un ouvert d'un espace affine de dimension~$n$, des données 
supplémentaires sont donc nécessaires.
Comme on vient de le voir, 
ce peut être
un système de coordonnées affines ;
mais comme la dépendance en un tel système
ne fait intervenir que la valeur absolue du jacobien
du changement de variables, 
il est suffisant de se donner un $n$-vecteur
et une orientation. 

Le but des paragraphes
qui suivent est d'expliciter cette
assertion un peu vague et de définir
du même coup l'intégrale d'une $(n-1,n)$-forme
sur le bord d'un demi-espace de~$V$. 

\subsection{}\label{sss.prelim-integration}
Soit $V$ un $\R$-espace affine de dimension finie,
soit $\mathscr B$ le faisceau des fonctions boréliennes sur~$V$.
Soit $W$ un sous-espace affine de~$V$, soit $p$ sa dimension;
soit~$o$ une orientation de~$W$.

Soit $n$ un entier naturel et soit~$\lambda\in\Lambda^n V$.
Soit $\omega$ une $(p,n)$-forme à coefficients boréliens
sur~$V$. Par définition, $\omega$ est un
élément de $\mathscr B(U)\otimes\Lambda^p \vec V^*\otimes \Lambda^n\vec V^*$.
On peut donc contracter la partie~$\Lambda^n\vec V^*$
de la forme~$\omega$ par~$\lambda$ de façon
à obtenir
une $(p,0)$-forme sur~$V$, c'est-à-dire
une $p$-forme sur~$V$, au sens usuel de la géométrie différentielle.

Lorsque $p=n$, $V=\R^n$, que $\omega$ est donnée
par la formule~\eqref{eq.omega.n-n-d'd''} et que $\lambda$
est le $n$-vecteur $e_1\wedge\dots\wedge e_n$, 
on obtient  $(-1)^{n(n-1)/2}$ fois la forme
$\omega^\sharp \mathrm dx_1\dots\mathrm dx_n$ 
qui est intégrée dans la relation~\eqref{eq.int.lagerberg}
(attention à l'ordre des termes~$\di x_i$ et~$\dc x_i$). 

On définit ainsi une $(p,0)$-forme
$ \langle\omega,\lambda\rangle$ sur~$V$ comme $(-1)^{n(n-1)/2}$
fois la contraction décrite ci-dessus.
Notons alors
\[ \int_W \langle\abs\omega,\abs\lambda\rangle \]
la masse totale de la mesure 
sur~$W$ définie
par la restriction à~$W$
de la $p$-forme~$\langle\omega,\lambda\rangle$,
si ladite restriction
est localement intégrable ; dans le cas contraire,
on pose
\[ \int_W \langle\abs\omega,\abs\lambda\rangle=+\infty.\]
On dit que $\omega$ est intégrable sur~$W$ si
\[ \int_W\langle\abs\omega,\abs\lambda\rangle<+\infty.\]
Dans ce cas, l'orientation~$o$ fixée sur~$W$ permet de définir
l'intégrale de~$\langle\omega,\lambda\rangle$ sur~$W$;
sous-entendant le choix de l'orientation~$o$, on note naturellement 
\[ \int_W \langle\omega,\lambda\rangle \]
cette intégrale. 

Les formes~$\omega$ de type~$(p,n)$ à coefficients boréliens 
sur~$V$
qui sont intégrables
sur~$W$ forment un sous-espace vectoriel de~$\mathscr B^{p,n}_V(V)$
sur lequel
l'application $\omega\mapsto  \int_W \langle\omega,\lambda\rangle $
est linéaire.

\subsection{}\label{sss.integrale.affine}
Lorsqu'on change l'orientation~$o$ en l'orientation opposée,
l'intégrale d'une forme (par rapport à~$(o,\lambda)$)
est transformée en son opposée.
De même, lorsqu'on change le $n$-vecteur~$\lambda$ en le multiple~$a\lambda$,
l'intégrale d'une forme est multipliée par~$a$.

En particulier, la notion d'intégrabilité et l'intégrale ne sont
pas modifiées si l'on change
simultanément l'orientation~$o$ en l'orientation opposée
et le $n$-vecteur~$\lambda$ en son opposé~$-\lambda$.
Ces notions ne dépendent donc que de la classe~$\mu$
du couple~$(o,\lambda)$ dans le produit contracté
$
\Lambda^n \vec V
\times ^ {\{-1,1\}}
\Or(\vec W)
$
de l'espace des $n$-vecteurs sur~$\vec V$
et 
de l'ensemble des orientations sur~$\vec W$, munis des actions naturelles
de~$\{-1,1\}$.

On notera ainsi $\int_W \langle\omega,\mu\rangle$
et $\int_W\langle\abs\omega,\abs\mu\rangle$
ce qui était noté
$\int_W \langle\omega,\lambda\rangle$
et $\int_W\langle\abs\omega,\abs\lambda\rangle$.
 
On peut penser aux symboles $\abs\omega$
et $\abs \mu$ comme représentant 
\og $\omega$ et $\mu$ au signe près\fg.

\subsection{Vecteurs-volume} 
Supposons que $n$ soit égal à la dimension de~$V$. 
Nous appellerons \emph{vecteur-volume}\index{vecteur-volume} sur~$V$
un élément du produit contracté $\abs{\Lambda^n\vec V}= 
\Lambda^n\vec V\times ^ {\{-1,1 \}} \Or(\vec V)$.
Concrètement, on obtient bien sûr un tel 
vecteur-volume $\abs{e_1\wedge\dots \wedge e_n}  $
à partir d'une base $(e_1,\ldots, e_n)$ 
de~$\vec V$, décrétée directe,
et du $n$-vecteur $e_1\wedge\dots \wedge e_n$ de~$\vec V$.

L'espace des vecteurs-volume sur~$V$ est naturellement
un~$\R$-espace vectoriel de dimension~$1$.

On dira qu'un vecteur-volume~$\mu$ est (strictement) positif s'il existe
une base $(e_1,\dots,e_n)$ de~$\vec V$ et un nombre réel (strictement)
positif~$a$ tel que $\mu=a\abs{e_1\wedge\dots \wedge e_n}$.
C'est alors le cas pour toute base de~$V$. 

\subsection{Intégrale des formes de type~$(n,n)$}
Supposons toujours que $n$ soit égal à la dimension de~$V$
et soit $\mu$ un vecteur-volume sur~$V$.
 
D'après les paragraphes précédents,
on peut  donc
définir : 
\begin{itemize}
\item L'intégrale 
$\int_V \langle\abs \omega,\abs \mu\rangle$
pour toute forme borélienne~$\omega$ de type $(n,n)$ sur $V$ ; 
\item La notion de $\mu$-intégrabilité d'une telle forme ; 
\item Pour une $(n,n)$-forme $\omega$ sur $V$ qui est 
borélienne et $\mu$-intégrable,
l'intégrale  $\int_V \langle \omega,\mu\rangle$.
\end{itemize}
Explicitons ces définitions en coordonnées.

Soit $(e_1,\dots,e_n)$ une base de~$\vec V$;
écrivons $\mu=a\abs{e_1\wedge\dots\wedge e_n}$.
Si $\omega=\omega^\sharp(x_1,\ldots, x_n) \di x_1\wedge\dc x_1\wedge \ldots \wedge \di x_n \wedge \dc x_n$,
alors
$\omega$ est $\mu$-intégrable si et seulement si 
la fonction~$a\omega^\sharp$ sur~$V$
est intégrable.
Cela  ne dépend d'ailleurs pas de~$a$, sauf si $a=0$ auquel cas
toute forme borélienne est $\mu$-intégrable.

On a 
\[ \int_V \langle \abs \omega ,\abs \mu\rangle = \abs a \int_V 
\abs{\omega^\sharp}(x_1,\dots,x_n)\, \mathrm dx_1\dots\mathrm dx_n, \]
et \[ \int_V \langle \omega ,\mu\rangle = a \int_V
\omega^\sharp(x_1,\dots,x_n)\, \mathrm dx_1\dots\mathrm dx_n\]
lorsque $\omega$ est $\mu$-intégrable.

Lorsque $a=1$, on retrouve la formule de Lagerberg.

Supposons que le vecteur-volume~$\mu$ soit positif. Avec les notations
précédentes, on a donc $a\geq 0$.
Ainsi, l'intégrale $\int_V \langle\omega,\mu\rangle$
de  toute $(n,n)$-forme positive~$\omega$
est positive
et  coïncide avec l'intégrale $\int_V \langle\abs\omega,\abs\mu\rangle$
de sa «valeur absolue».

\subsection{}
Soit $f\colon V\ra V'$ un isomorphisme d'espaces affines de dimension finie
et soit $\mu$ un vecteur-volume sur~$V$. 
Soit  $f_*(\mu)$ le vecteur 
volume sur~$V'$ 
déduit de~$\mu$ par transport de structure.
Si $\mu$ est positif, $f_*(\mu)$ l'est aussi.

Soit $\omega'$ une $(n,n)$-forme à coefficients boréliens
sur~$V'$ et soit $\omega=f^*\omega'$.
Pour que $\omega$ soit $\mu$-intégrable, il faut et il suffit que
$\omega'$ soit $f_*(\mu)$-intégrable; alors,
 \[ 
\int_{V}\langle f^*\omega ,\mu\rangle =  \int_{V'} \langle 
\omega,{f_*(\mu)} \rangle.\]

\subsection{Le cas des $(n-1,n)$-formes : les intégrales de bord}
\label{sss.integrale.bord}
On suppose toujours que $n$ est la dimension de $V$. 
Soit $V_+$ un demi-espace fermé de~$V$; 
son bord~$\partial(V_+)$ est un hyperplan affine de~$V$
et $V_+$ est l'une des deux composantes connexes
de~$V\setminus\partial(V_+)$.
Si $\ell$ est une forme affine telle que $V_+=\{\ell\geq 0\}$,
on dit qu'un vecteur~$v\in\vec V$ est  sortant
de~$V_+$ si $\ell(v) \leq  0$, 
l'adverbe strictement signifiant que l'inégalité précédente
est stricte.

Tout vecteur-volume~$\mu$ sur~$V$ induit canoniquement un 
élément $\partial \mu_+$ 
de $\Or(\vec {\partial (V_+)})\times^ {\{-1,1\}}\Lambda^n \vec V$
défini comme suit. On choisit une orientation $o$ de $\vec{\partial (V_+)}$, 
une base directe $(e_2,\ldots,e_n)$ de $\vec{\partial (V_+)}$, 
un vecteur~$e_1$ strictement sortant de $V_+$, 
et l'on note~$o'$ l'orientation de~$V$ 
pour laquelle la base~$(e_1,\ldots, e_n)$ est directe.
(On dira que $V$ est orienté suivant \emph{la règle de la normale sortante}
à partir de l'orientation de~$\vec{\partial (V_+)}$.)
Il existe alors un unique $\lambda\in \Lambda^n\vec V$ tel que $(\lambda,o')$ 
soit égal à $\mu$ dans $\abs{\Lambda^n\vec V}$. Le couple $(\lambda,o)$ 
définit un élément $\partial \mu_+$ 
de $\Lambda^n \vec V\times^ {\{-1,1\}}\Or(\vec {\partial (V)})$
qui est indépendant des choix faits. 

D'après les paragraphes~\ref{sss.prelim-integration} 
et~\ref{sss.integrale.affine},
on peut alors définir : 
\begin{itemize}
\item L'intégrale 
$\int_{\partial (V_+)} \langle\abs \omega,\abs{ \partial \mu_+}\rangle$
pour toute forme borélienne
$\omega$ de type $(n-1,n)$ sur~$V$ ; 
\item
La notion de $\partial \mu_+$-intégrabilité d'une telle forme
sur~$\partial (V_+)$ ; 
\item
Pour une forme $\omega$ sur~$V$ qui est mesurable et
$\partial\mu_+$-intégrable
sur~$\partial (V_+)$, l'intégrale 
$\int_{\partial (V_+)} \langle \omega,\partial \mu_+\rangle$. 
\end{itemize}
Concrètement, fixons une origine de~$V$ appartenant à~$\partial (V_+)$
et choisissons $(e_1,\dots,e_n)$ comme ci-dessus:
c'est une base directe de~$V$; 
si $(x_1,\dots,x_n)$ est la base duale de~$\vec V$, 
le demi-espace $V_+$ est défini par l'inégalité $x_1\leq 0$;
enfin, $(e_2,\dots,e_n)$ est une base directe de~$\partial (V_+)$.
Compte tenu des orientations choisies, 
le vecteur-volume~$\mu$ est donné par un $n$-vecteur 
$a e_1\wedge\dots\wedge e_n$ de~$\Lambda^n \vec V$. 
En coordonnées, on peut donc écrire
\[ \omega = \sum_{i=1}^n
  \omega_{i} \di x_{1}\wedge\dc x_1\wedge \dots \wedge \widehat{\di x_i} 
\wedge\dc x_i
\wedge\dots \di x_{n}\wedge \dc x_n. \]
On a alors
\[ \int_{ \partial (V_+)} \langle \abs \omega ,\abs {\partial \mu_+}\rangle = \abs a
\int_{\partial (V_+)}
 \abs{\omega_{1}}(0,x_2,\dots,x_n)\mathrm dx_2\dots\mathrm 
dx_n.\] La forme $\omega$
est $\partial \mu_+$-intégrable
sur $W$ si et seulement si 
$a\omega_1(0,*,\ldots, *)$ est $L^1$ (ce qui ne dépend pas de $a$, sauf si celui-ci est nul, auquel cas
toute forme borélienne est $\partial \mu_+$-intégrable). 
Si c'est le cas,
\[ \int_{\partial (V_+)} \langle \omega ,\partial \mu_+\rangle = a \int_{\partial (V_+)} 
\omega_{1}(0,x_2,\dots,x_n)
\mathrm dx_2\dots\mathrm dx_n. \]



\begin{lemm}[Formule de Stokes]
\label{lemm.stokes-affine}
Soit $\mu\in\abs{\Lambda^n \vec V}$.
Soit $V_+$ un demi-espace fermé dans~$V$.
Soit $\omega$ une forme de type~$(n-1,n)$ sur~$V$, à coefficients 
de classe~$\mathscr C^1$,
telle que $\supp(\omega)\cap V_+$  est compact.
On a
\[ \int_{V_+} \langle  \di\omega,\mu\rangle
 = \int_{\partial (V_+)} \langle \omega, \partial \mu \rangle.\]
\end{lemm}
\begin{proof}
Cette formule n'est qu'une récriture de la formule de Stokes usuelle.
Vérifions-le en coordonnées. Supposons donc que 
$V$ est l'espace~$\R^n$ muni de son orientation canonique
et que $V_+=\R_-\times\R^{n-1}$. Notons $(e_1,\dots,e_n)$
la base canonique de~$\R^n$. Le vecteur-volume~$\mu$
correspond à un $n$-vecteur $ae_1\wedge\dots\wedge e_n$.
D'autre part, on peut écrire
\[ \omega = \sum_{i=1}^n
  \omega_{i} \di x_{1}\wedge\dots\wedge\widehat{\di x_i}\wedge
\dots\wedge \di x_n \wedge \dc x_1\wedge \dots \wedge \dc x_n. \]
Par définition, on a 
\begin{align*}
 \di\omega & = \sum_{i=1}^n \sum_{j=1}^n\frac{\partial\omega_i}{\partial x_j} 
\di x_j\wedge\di x_{1}\wedge\dots\wedge\widehat{\di x_i}\wedge
\dots\wedge \di x_n \wedge \dc x_1\wedge \dots \wedge \dc x_n \\
& = \sum_{i=1}^n (-1)^{i-1} \frac{\partial\omega_i}{\partial x_i} 
\di x_1 \wedge\dots\wedge\di x_n \wedge \dc x_1\wedge \dots \wedge \dc x_n.
\end{align*}
Soit $\alpha$ la forme différentielle de degré~$n-1$ sur~$V$
donnée par
\[ \alpha = (-1)^{n(n-1)/2} \sum_{i=1}^n \omega_i \mathrm dx_1\wedge\dots\wedge
\widehat{\mathrm dx_i}\wedge\dots\wedge\mathrm dx_n. \]
Soit $\beta$ la $n$-forme sur~$V$ donnée par
\[ \beta = (-1)^{n(n-1)/2} \sum_{i=1}^n(-1)^{i-1} \frac{\partial\omega_i}{\partial x_i}
 \mathrm dx_1\wedge\dots\wedge\mathrm dx_n. \]
On a donc
\[ \int_{\partial (V_+)} \langle \omega ,\partial \mu_+\rangle = 
a \int_{\partial (V_+)} \alpha 
\quad\text{et}\quad 
\int_{V_+} \langle \di\omega,\mu_+\rangle
= a \int_{V_+} \beta. \]
D'après la formule de Stokes usuelle, 
$\int_{\partial (V_+)}\alpha=\int_{V_+}\mathrm d\alpha$.
Or, 
\[ \mathrm d\alpha = (-1)^{n(n-1)/2} \sum_{i=1}^n (-1)^{i-1}
\frac{\partial\omega_i}{\partial x_i}
\mathrm dx_1\wedge\dots\wedge \mathrm dx_n = \beta, \]
d'où le lemme.
\end{proof}
\begin{lemm}[Formule de Green]
\label{lemm.green-affine}
Soit $\mu\in\abs{\Lambda^n \vec V}$.
Soit $V_+$ un demi-espace fermé dans~$V$.
Soit $\alpha$ et $\beta$ des formes \emph{symétriques} de type $(p,p)$
et $(q,q)$ sur~$V$, avec $p+q=n-1$,
à coefficients de classe~$\mathscr C^2$; si l'intersection
de leurs supports est compacte,  on a
\[ \int_{V_+} \langle  \alpha\wedge \ddc\beta - \ddc\alpha\wedge\beta 
,\mu\rangle
 = \int_{\partial (V_+)} \langle \alpha\wedge \dc\beta - 
\dc\alpha\wedge\beta,\partial \mu^+ \rangle.\]
\end{lemm}
\begin{proof}
Posons
\[ \omega=\alpha\wedge \dc\beta-\dc\alpha\wedge\beta. \]
D'après la formule de Stokes (lemme~\ref{lemm.stokes-affine}),
on a 
\[ \int_{V_+} \langle\di\omega ,\mu\rangle= \int_{\partial (V_+)}\langle\omega,
\partial\mu_+\rangle. \]
Par ailleurs, comme $\alpha$ est de bidegré~$(p,p)$,
\[ \mathrm d\omega=(\alpha\wedge\ddc\beta-\ddc\alpha\wedge\beta)
+ (\di\alpha\wedge\dc\beta+\dc\alpha\wedge\di\beta). \]
D'après le lemme~\ref{lemm.didc=-dcdi}, on a 
$\di\alpha\wedge\dc\beta=-\dc\alpha\wedge\di\beta$,
d'où le lemme.
\end{proof}

\section{Supercourants}

Soit $V$ un espace affine réel de dimension finie, $n$,
muni d'un vecteur-volume strictement positif~$\mu$.

\subsection{}
Dans~\cite{lagerberg2012}, Lagerberg munit l'espace des superformes
de type~$(p,q)$ sur~$V$,
lisses à support contact,
d'une topologie localement convexe, 
similaire à celle des fonctions lisses à support compact.
Les éléments de son dual topologique sont appelés supercourants,
voire courants, de bidimension~$(p,q)$ ou de bidegré~$(n-p,n-q)$.

\subsection{}
L'intégrale des $(n,n)$-formes
sur~$V$ qui lui est associée fournit une application
$\alpha\mapsto [\alpha]$
de l'espace des~$(p,q)$-formes dans celui des  courants
de bidegré~$(p,q)$
par la formule 
\[ \langle [\alpha],\beta\rangle= \int_V \langle \alpha\wedge\beta,\mu\rangle.
\]
Cette application est injective.

\subsection{}
Grâce à la formule de Stokes,
le calcul différentiel des superformes se prolonge par dualité
en un calcul différentiel des supercourants.

\subsection{}
Comme dans le cas des variétés, on peut écrire un courant sur~$\R^n$
comme une « forme à coefficients distributions »:
si $T$ est de bidegré~$(p,q)$, on a une décomposition 
\[ T = \sum_{IJ} T_{IJ} \di x_I\wedge \dc x_J \]
indexée par les couples $(I,J)$ de parties de~$\{1,\dots,n\}$
tels que $\Card(I)=p$ et $\Card(J)=q$,
où les $T_{IJ}$ sont des courants de bidegré~$(0,0)$.
À un signe près qui ne dépend que du couple~$(I,J)$, 
on  a $\langle T_{IJ},\phi\eta\rangle = \langle T, \phi \di x^I\wedge \dc x^J\rangle$.
Remarquons que $T_{IJ}\wedge\eta$ est une distribution sur~$\R^n$.

\subsection{}
On dit qu'un courant de bidimension $(p,p)$ est positif
(\resp faiblement positif, \resp fortement positif) s'il est
symétrique et si l'image de toute forme positive
(\resp fortement positive, \resp faiblement positive) est positive.

Si $f$ est une fonction (continue) sur~$V$, 
le courant $\ddc [f]$ est positif si et seulement si
la fonction~$f$ est convexe;
inversement, tout courant positif fermé de bidegré~$(1,1)$
est de cette forme
(propositions~2.5 et~2.6 de~\cite{lagerberg2012}).

\subsection{}\label{ss.lagerberg-convolution}
Lageberg~\cite{lagerberg2012} explique comment,
par convolution, les supercourants sur~$\R^n$ peuvent
être approchés par des formes lisses. 
Rappelons la construction.

Soit $\rho$ une fonction lisse à support compact sur~$\R^n$;
pour tout $\eps>0$, posons $\rho_\eps(x)=\eps^{-n}\rho(x/\eps)$.

Soit $f$ une fonction lisse sur un ouvert~$\Omega$ de~$\R^n$.
La fonction $\rho*f$ sur~$\R^n$ est définie par
\[ \rho*f(x) = \int_\Omega f(y) \rho(x-y)\,dy. \]
Son support est contenu dans $\supp(f)+\supp(\rho)$.
Si $\rho$ et $f$ sont positives, alors $\rho*f$ est positive.

Si $\Omega=\R^n$, on écrit
\[ \rho*f(x) = \int_{\R^n} \rho(y) f(x-y)\, dy, \]
On en déduit que si $f$ est lisse, alors $\rho*f$ est lisse,
et ses dérivées sont données par dérivation sous le signe somme.
De plus, si $\rho$ est positive d'intégrale~$1$,
alors $\rho_\eps*\omega$ converge uniformément vers~$\rho$
sur tout compact de~$\R^n$ lorsque $\eps\to0$, 
et de même pour chacune de ses dérivées.
De même, si $f$ est croissante par rapport à~$x_i$,
il en est de même de~$\rho*f$.

Dans le cas général,
on peut écrire
\[ \rho*f(x) = \int_{\R^n} \rho(y) \tilde f(x-y)\, dy, \]
où $\tilde f$ est la fonction (mesurable) sur~$\R^n$ égale à~$f$ sur~$\Omega$
et à~$0$ ailleurs. De cette formule, on en déduit que des variantes
des propriétés précédentes.
Si $\rho$ est positive, de support contenu dans~$B$
et si $f$ est croissante par rapport à~$x_i$ (resp. convexe, resp. lisse)
sur un ouvert~$\Omega'$ de~$\Omega$ tel que $\Omega'+\supp(\rho)\subset\Omega$, 
alors $\rho*f$ est croissante par rapport à~$x_i$ (resp. convexe, resp. lisse)
sur cet ouvert~$\Omega'$.
De même, si $\rho$ est positive d'intégrale~$1$,
alors $\rho_\eps*\omega$ converge vers~$\rho$ lorsque $\eps\to 0$,
uniformément sur tout compact de~$\Omega$, avec ses dérivées.

Soit maintenant $\omega$ une forme lisse sur~$\R^n$.
Soit $\rho$ une fonction lisse à support compact sur~$\R^n$;
on définit la forme $\rho*\omega$ en se ramenant au cas des fonctions
sur~$\R^n$ en décomposant 
$\omega = \sum_{I,J} \omega_{IJ} \di x_I \wedge \dc x_J$
et en posant
\[ \rho* \omega = \sum_{IJ} (\rho * \omega_{IJ}) \di x_I \wedge \dc x_J.\] 

Lorsque $\omega$ est à support compact, la forme $\rho*\omega$
est à support compact.

Par dérivation sous le signe somme, on a 
\[ \di (\rho*\omega) = \rho*\di\omega, \qquad \dc (\rho*\omega)=\rho*\dc \omega. \]

Supposons que  $\rho$ soit positive.
Alors si $\omega$ est positive (resp. faiblement positive, resp. fortement positive), il en est de même de $\rho*\omega$.

Supposons que $\rho$ soit positive, d'intégrale~$1$;
pour tout $\eps>0$, posons $\rho_\eps(x)=\eps^{-n}\rho(x/\eps)$.
Alors $\rho_\eps*\omega$ converge uniformément vers~$\rho$
sur tout compact de~$\R^n$, avec ses dérivées.

Soit maintenant $T$ un courant de bidimension~$(p,q)$ sur~$\R^n$.
La convolution $\rho * T$  est un courant de bidimension~$(p,q)$
défini par
\[ \langle \rho* T, \omega\rangle = \langle T, \rho* \omega \rangle \]
pour toute forme~$\omega$, lisse de type~$(p,q)$  et à support compact.
En fait, c'est un courant associé à la \emph{forme} lisse
\[ \sum _{I,J} (\rho*(T_{IJ}\wedge\eta)) \di x_I \wedge \dc x_J,\]
puisque la convolée de~$\rho$ et de la distribution $T_{IJ}\wedge\eta$
est une fonction lisse.

On a $\di (\rho*T)=\rho*\di T$ et $\dc(\rho*T)=\rho*\dc T$.

Supposons $\rho$ positive.
Si $T$ est positif (resp. faiblement positif, resp. fortement positif),
il en est de même de $\rho* T$.

Supposons que $\rho$ est positive d'intégrale~$1$.
Alors $\rho_\eps*T$ converge faiblement vers~$T$ quand $\eps$ tend vers~$0$.

En particulier, tout courant positif fermé est limite de formes lisses positives fermées.

\subsection{}
Inspiré par les théories de Demailly~\cite{demailly1985}
et Bedford-Taylor~\cite{bedford-t82}, Lagerberg définit,
par approximation, une notion
de produit de courants positifs de la forme $\ddc f$,
où $f$ est une fonction convexe sur~$\R^n$.

Le courant $\MA(f)=(\ddc f)^n $ est alors une mesure positive
sur~$V$, qu'on appelle la mesure de Monge-Ampère de~$f$.

\section{Un calcul d'opérateur de Monge-Ampère}

\begin{lemm}\label{lemm.MA-H}
Soit $U$ un ouvert de~$\R^n$ et 
soit $h\colon U\ra \R$ une fonction de classe~$\mathscr C^2$.
Alors,
\[ (\ddc h)^n = n! \det\big( \frac{\partial^2 h}{\partial x_i\partial x_j}\big)
   \di x_1\wedge\dc x_1\wedge\dots\wedge\di x_n\wedge\dc x_n. \]
\end{lemm}
\begin{proof}
Partons de la formule
\[ \ddc h = \sum_{i,j=1}^n \frac{\partial^2 h}{\partial x_i\partial x_j}
\di x_i \wedge \dc x_j. \]
Alors, dans l'expression 
\[  (\ddc h)^n   = \sum_{\substack{i_1,\dots,i_n \\ j_1,\dots,j_n}}
 \prod_{k=1}^n\frac{\partial^2 h}{\partial x_{i_k}\partial x_{j_k}}
\di x_{i_k} \wedge \dc x_{j_k} \]
(rappelons que les formes de
bidegré~$(1,1)$ commutent deux à deux),
tous les termes pour lesquels $\{i_1,\dots,i_n\}\neq \{1,\dots,n\}$
ou $\{j_1,\dots,j_n\}\neq\{1,\dots,n\}$ sont nuls.
Ainsi,
\[  (\ddc h)^n   = \sum_{\sigma,\tau\in\mathfrak S_n}
 \prod_{k=1}^n\frac{\partial^2 h}{\partial x_{\sigma(k)}\partial x_{\tau(k)}}
\,
\di x_{\sigma(1)} \wedge \dc x_{\tau(1)}\wedge 
\dots \wedge \di x_{\sigma(n)} \wedge \dc x_{\tau(n)}.
\]
Remarquons aussi que
\[ \di x_{\sigma(1)} \wedge \dc x_{\tau(1)}\wedge 
\dots \wedge \di x_{\sigma(n)} \wedge \dc x_{\tau(n)}
=  \eps(\sigma)\eps(\tau)
\di x_1\wedge \dc x_1\wedge \dots \wedge \di x_n\wedge \dc x_n, \] 
où $\eps(\cdot)$ désigne la signature
d'une permutation.
En outre, si $\psi=\tau\circ\sigma^{-1}$,  on a
 \[
\prod_{k=1}^n\frac{\partial^2 h}{\partial x_{\sigma(k)}\partial x_{\tau(k)}}
= \prod_{k=1}^n 
\frac{\partial^2 h}{\partial x_{k}\partial x_{\psi(k)}}
\]
et 
\[  \eps(\sigma)\eps(\tau)=\eps(\psi). \]
Par suite,
\[ (\ddc h)^n = n! \sum_{\psi\in\mathfrak S_n}
 \prod_{k=1}^n \frac{\partial^2 h}{\partial x_{k}\partial x_{\psi(k)}}
 \eps(\psi)\,
\di x_1\wedge \dc x_1\wedge \dots \wedge \di x_n\wedge \dc x_n, \] 
ainsi qu'il fallait démontrer.
\end{proof}

\begin{lemm}
Munissons~$\R^n$ de la mesure de Lebesgue et de son vecteur-volume
canonique~$\mu$.
Soit $U$ un ouvert convexe de~$\R^n$ et 
soit $h\colon U\ra \R$ une fonction lisse, strictement convexe.
L'application gradient $\nabla h\colon U\ra \R^n$ est injective
et $\int_U \langle (\ddc h)^n ,\mu\rangle = n! \vol\nabla h(U)$.
\end{lemm}
\begin{proof}
Soit $x$ et~$y$ des points distincts de~$ U$. 
Comme la fonction $f\colon t\mapsto h(x+t(y-x))$ 
sur $[0,1]$ est strictement convexe, on a  $f'(1)>f'(0)$ puis
$\langle\nabla h(y)-\nabla h(x),y-x\rangle>0$.
En particulier, $\nabla h(x)\neq\nabla h(y)$.

La matrice jacobienne $\Jac(\nabla h)$ 
de l'application $\nabla h\colon U\ra\R^n$
n'est autre que la matrice hessienne~$\mathscr H(h)$ de l'application~$h$.
Puisque~$h$ est convexe, on a $\det(\mathscr H(h))\geq 0$.
D'après le lemme précédent, le théorème de Sard et
la formule de changement de variables dans les intégrales
multiples, on a donc
\[ \int_U \langle (\ddc h)^n , \mu \rangle = n! \int_U \det\mathscr H(h)
= n! \int_U \det\Jac(\nabla h) =  n! \vol(\nabla h(U)).
\qedhere\]
\end{proof}

\begin{lemm}
Soit $V$ un espace affine réel de dimension finie
et soit $(f_i)_{i\in I}$ une famille non vide de formes affines
sur~$V$. Soit $j\in I$. 
Les conditions suivantes sont équivalentes:
\begin{enumerate}
\def\theenumi{\roman{enumi}}\def\labelenumi{(\theenumi)}

\item Pour toute forme affine $g$ sur $V$, 
l'application affine $(f_i-g)_i$ de $V$ dans $\R^I$ est injective. 

\item L'application affine $(f_i-f_j)_i$ de $V$ dans $\R^I$ est injective. 

\item L'application affine $(f_i)_i$ de $V$ dans $\R^I$ est injective, 
et son image ne contient aucune droite parallèle à la droite diagonale
engendrée par le vecteur constant $(1)_i$. 

\item La famille $(\vec{f_i}-\vec{f_j})_i$
contient une base de l'espace dual $\vec V^*$. 

\item Le sous-espace affine de l'espace dual $\vec V^*$ engendré par les
$\vec{f_i}$ est égal à $\vec V^*$.

\end{enumerate}

\end{lemm}

\begin{proof}
Il est clair que (i)$\Rightarrow$(ii). 
Supposons que (ii) est vraie. Posons $f=(f_i)_i$. L'application affine 
$(f_i-f_j)_i$ étant injective, la famille
$(\vec{f_i}-\vec{f_j})_i$ engendre $\vec V^*$ ; il s'ensuit que
$(\vec {f_i})_i$ engendre $\vec V^*$ et donc
que $f$ est injective. 
Supposons qu'il existe une famille finie
$(a_i)_i$ de nombres réels, deux points $v$ et $w$ de $V$, et deux nombres
réels distincts $s$ et $t$ tels que $f(v)=(a_i+s)_i$ et $f(w)=(a_i+t)_i$ 
(comme $s\neq t$ on a nécessairement $v\neq w$). 
L'application affine
$(f_i-f_j)_i$ prend alors la valeur $(a_i-a_j)_i$ en $v$ et en $w$, ce qui 
contredit (i). Par conséquent, $f(V)$ ne contient aucune droite parallèle
à la diagonale, d'où (iii). 

Supposons que (iii) est vraie, et montrons (iv).  Si la famille
$(\vec{f_i}-\vec {f_j})$ n'était pas de rang maximal, il existerait un vecteur
$\vec v$ non nul dans l'intersection des noyaux des formes linéaires
$\vec{f_i}-\vec {f_j}$. L'image $\vec f(\vec v)$ serait
alors un vecteur nul dont toutes les composantes
sont égales, c'est-à-dire un vecteur directeur
de la droite diagonale, si bien que $f(V)$ contiendrait une droite affine
parallèle à la diagonale, contredisant (iii). Par conséquent (iv) est vraie.

L'implication (iv)$\Rightarrow$(v) provient du fait que
le sous-espace affine de $\vec V^*$ engendré par les $\vec {f_i}$ est le 
sous-espace affine de $\vec V^*$ passant par $\vec{f_j}$ et dirigé par 
$\langle \vec{f_i}-\vec {f_i}\rangle_i$. 

Supposons enfin que (v) est vraie. Soit $g$ une forme affine sur $V$. 
La famille $(\vec{f_i}-\vec g)_i$ contient alors une base de $\vec V^*$ : sinon, 
les $\vec {f_i}$ appartiendraient toutes à un même sous-espace affine strict de 
$\vec V^*$ passant par $\vec g$. L'application linéaire $(\vec{f_i}-\vec g)_i$ 
est donc injective, et l'application affine $(f_i-g)_i$
l'est alors aussi, ce qui montre (i) et termine la preuve du lemme. 
\end{proof}

\begin{defi}\label{defi.tres-separante}
Soit $V$ un espace affine réel de dimension finie
et soit $(f_i)_{i\in I}$ une famille non vide de formes affines
sur~$V$. 
Nous dirons que la famille~$(f_i)$ est \emph{très séparante}
si elle satisfait les conditions équivalentes du lemme ci-dessus.
\end{defi}

\subsection{}
\label{sss-tres-separante}
Soit $V$ un espace affine réel de dimension finie
et soit $(f_i)_{i\in I}$ une famille non vide de formes affines
sur~$V$. 

\begin{enumerate}\def\theenumi{\alph{enumi}}\def\labelenumi{\theenumi)}
\item Pour toute forme affine~$g$ sur~$V$,
la famille $(f_i-g)_i$ est très séparante si et seulement
s'il en est de même de la famille~$(f_i)_i$ : cela résulte immédiatement
du critère (i).

\item Si la famille $(f_i)$ est très séparante, 
l'ensemble des points~$x$ de~$V$ tels
que $f_i(x)=f_j(x)$  pour tous $i,j$ est de cardinal au plus~$1$.
(En effet, son image par l'application affine injective $f$ 
est l'intersection de~$f(V)$ et de la diagonale de~$\R^I$.)

\item Si la famille $(f_i)_i$ contient une forme constante,
elle est très séparante si et seulement si l'application affine~$f$
est injective. (Prendre~$j$ tel que
$f_j$ soit constante et appliquer le critère (ii)). 
\end{enumerate}

%

\begin{lemm}\label{lemm.approx-max}
Soit $V$ un espace vectoriel réel de dimension finie,
soit $(f_i)_{i\in I}$ une famille finie non vide de formes affines sur~$V$;
soit $f=(f_i)\colon V\to\R^I$.
Pour tout $\eps>0$, on pose $g_\eps(x)=\eps \log\big(\sum_{i\in I} \exp(f_i(x)/\eps)\big)$ ; on pose aussi $g(x)=\sup_{i\in I} f_i(x)$.

\begin{enumerate}
\item
La fonction $g\colon V\to\R$ est convexe.

\item
Pour tout $\eps>0$, la fonction $g_\eps\colon V\to\R$ 
est convexe ;
elle est strictement convexe si et seulement si la famille~$(f_i)$
est très séparante.
\item
Lorsque $\eps\to 0^+$, les fonctions $g_\eps$ tendent vers~$g$, 
uniformément et en décroissant.
\end{enumerate}
\end{lemm}
\begin{proof}
\begin{enumerate}
\item
La fonction~$g$ est convexe puisque c'est un maximum
de fonctions convexes.

\item
Il suffit de traiter le cas où $\eps=1$.
Commençons par traiter le cas où $V=\R^n$ et les~$f_i$
sont les fonctions coordonnées.
La fonction~$g_1$ est lisse sur~$\R^n$; pour démontrer
qu'elle est convexe, il suffit de démontrer
que sa hessienne~$D^2_x g_1$ est positive en tout point~$x$. Or,
\[ \frac{\partial^2}{\partial x_i\partial x_j} g_1
= \frac\partial{\partial x_i} \frac{ \exp(x_j)}{\sum  \exp(x_k)}
= \frac{ \exp(x_i)}{\sum  \exp(x_k)}\delta_{ij}
- \frac{ \exp(x_i)\exp(x_j)}{(\sum  \exp(x_k))^2}. \]
Par suite, pour tout $x\in\R^n$ et tout $y\in\R^n$,  on a
\[
D^2_xg_1(y,y)  = \sum_{i=1}^n 
 \frac{ y_i^2 \exp(x_i)}{\sum  \exp(x_k)}
- \sum_{i,j=1}^n
\frac{y_i  y_j\exp(x_i)\exp(x_j)}{(\sum  \exp(x_k))^2}  .\]
Posons $u_i =  \exp(x_i)$. Comme la fonction $y\mapsto y^2$
est strictement convexe, on a 
\[ (\sum u_i y_i^2)(\sum u_i) \geq (\sum u_i y_i)^2,\]
avec égalité si et seulement si les~$y_i$ sont tous égaux.
On a donc $D^2_xg_1(y,y)\geq 0$ et $D^2_x g_1(y,y)=0$ si et seulement
si les~$y_i$ sont tous égaux.

Dans le cas général, la fonction~$g_1$ est la composée de l'application
affine~$f$ 
et de la fonction sur~$\R^I$ analogue à celle que nous venons d'étudier.
Elle est donc convexe. De plus, sa hessienne est positive, et son noyau
est l'image inverse de la diagonale par~$\vec f$. Si la famille~$(f_i)$
est très séparante, la hessienne de~$g_1$ est donc définie positive
en tout point, donc~$g_1$ est strictement convexe.
Inversement, on laisse au lecteur le soin de vérifier que 
si $g_1$ est strictement convexe, alors la famille~$(f_i)$ est très séparante.

\item
Il suffit de traiter le cas où $V=\R^n$ et les $f_i$ sont les fonctions
coordonnées.
Soit $x\in V$. Pour tout $\eps>0$, on a l'inégalité
\[    g(x) \leq \eps g_\eps(x) \leq \eps \log(n)+g(x), \]
ce qui entraîne que
$g_\eps(x)$ converge uniformément vers~$g(x)$
lorsque $\eps$ tend vers~$0$ par valeurs supérieures.

Par ailleurs, si $a$ et $b$ sont des nombres réels tels que $a<b$,
vérifions l'inégalité $g_a(x)\leq g_b(x)$.
Fixons $x\in V$ et, pour $t>0$, 
posons $h(t)= g_t(x)=\frac1t \log(\sum \exp(tx_i))$;
il suffit de démontrer qu'elle décroît.
La fonction~$h$ est dérivable et l'on a
\[ h'(t) = -\frac1{t^2} \log(\sum e^{t x_i})
 + \frac 1 t\, \frac{ \sum x_i e^{t x_i}}{\sum e^{t x_i}}. \]
Puisque $tx_j \leq \log(\sum_i e^{tx_i})$ pour tout~$j$,  on a
\[ \sum_i t x_i e^{t x_i} 
 \leq \log (\sum_i e^{t x_i}) \sum_i e^{tx_i},\]
d'où l'inégalité $h'(t)\leq0$. Cela prouve que $h$ est décroissante. 
\qedhere
\end{enumerate}
\end{proof}

\begin{prop}\label{MA-delta_0/1}
Soit $V$ un espace vectoriel réel de dimension finie~$n$,
soit $(f_1,\dots,f_m)$ une famille non vide de formes linéaires sur~$V$
et soit $\mu$ un vecteur-volume strictement positif sur~$V$.
Soit $D$ une partie convexe de~$V$ contenant~$0$ et soit $W(D)$
le cône tangent de~$D$ en~$0$.
Posons $g=\max(f_1,\dots,f_m)$; pour tout $\eps>0$,
posons $g_\eps=\eps\log(\sum \exp(f_i(x)/\eps))$.
Alors, $W(D)$ est de masse finie~$\lambda_D$ pour la mesure positive 
(absolument continue par rapport à la mesure de Lebesgue)
$\langle(\ddc g_1)^n,\mu\rangle$ sur~$V$ et l'on a
\[ \int_D \psi \langle (\ddc g_{\eps})^n,\mu\rangle 
\to \lambda_D \psi(0)  \]
pour toute fonction continue bornée~$\psi$ sur~$D$.

De plus, si $\dim D=$ alors
$\lambda_D>0$ si et seulement si la famille $(f_1,\dots,f_m)$
est très séparante.
\end{prop}
\begin{proof}
Supposons d'abord que $D=V$; dans ce cas, $W(D)=V$.
Rappelons que $g_\eps$ 
est une fonction convexe, lisse, qui converge uniformément
vers $g$ quand $\eps\ra 0^+$. 
On a donc $\langle(\ddc g)^n,\mu\rangle 
= \lim_{\eps\ra 0^+} \langle (\ddc g_\eps)^n,\rangle$ au sens des mesures.
Soit $T$ l'homothétie de rapport~$1/\eps$; 
son jacobien  vaut~$1/\eps^n$
et l'on a $g_\eps=\eps T^* g_1$,
donc
\begin{align*}
\int_V \psi \langle (\ddc g_\eps)^n,\mu\rangle
& = \eps^n \int_V \psi \langle  T^*(\ddc g_1)^n,\mu \rangle \\
& = \int_V (T^{-1})^* \psi \langle (\ddc g_1)^n,\mu\rangle  \\
& = \int_V \psi (\eps x) \langle (\ddc g_1)^n,\mu\rangle 
\end{align*}
et 
\[ 
\int_V \psi \langle (\ddc g)^n,\mu\rangle
= \lim_{\eps\to 0^+} \int_V \psi(\eps x) \langle (\ddc g_1)^n,\mu\rangle.
\]
Lorsque $\eps$ tend vers~$0$, la fonction continue~$\psi(\eps  x)$ converge 
vers la fonction constante de valeur~$\psi(0)$
et nous allons déduire du théorème de convergence dominée que
\[ \int_V \psi \langle (\ddc g)^n,\mu\rangle = \psi(0)
\int_V \langle (\ddc g_1)^n,\mu\rangle. \]
%

La fonction~$g_1$ est convexe,
donc la mesure $\langle (\ddc g_1)^n,\mu\rangle $ sur~$V$ est positive.
Pour appliquer le théorème de convergence dominée, il suffit donc
d'établir la finitude de cette mesure.
Si elle pourrait être établie directement
(soit par un calcul, soit par des propriétés
générales, \cf \cite{lagerberg2012}, proposition~3.9),
cette finitude  résulte aussi du calcul précédent.
En effet, prenons pour~$\psi$ une fonction lisse à support compact,
positive ou nulle, et égale à~$1$
dans un voisinage~$B$ de~$0$. 
Alors,
\begin{align*}
\int_V  \psi \langle (\ddc g)^n,\mu\rangle 
&= \lim_{\eps\ra0^+}\int_{V}\psi(\eps x)\langle (\ddc g_1)^n ,\mu\rangle\\
& \geq \limsup_{\eps\ra 0^+}\int_{\eps^{-1}B}\langle (\ddc g_1)^n ,\mu\rangle \\
& = \int_V \langle (\ddc g_1)^n,\mu\rangle,
\end{align*}
d'où la finitude requise.

Le cas d'une partie convexe~$D$ quelconque s'en déduit. 
En effet, pour toute fonction continue
bornée~$\psi$ sur~$D$, on peut écrire
\[ \int_D \psi \langle (\ddc g_\eps)^n,\mu\rangle
 = \int_{V} \chi_{D}(\eps x)\psi(\eps x) \langle (\ddc g_1)^n,\mu\rangle. \]
Soit $x\in V$.
Si $x\in W(D)$, alors $\eps x\in D$ pour $\eps>0$ assez petit, 
et $\psi(\eps x)\to \psi(0)$; sinon, on a $\chi_D(\eps x)=0$ pour tout~$\eps$.
Ainsi, $\chi_D (\eps x)\psi(\eps x)\to \psi(0) \chi_{W(D)}(x)$.
Par convergence dominée, on a donc
\[ \int_D \psi \langle (\ddc g_\eps)^n,\mu\rangle
\to \psi(0)\int_{W(D)} \langle (\ddc g_1)^n,\mu\rangle
=\lambda_D \psi(0). \]

Supposons que la famille $(f_1,\dots,f_m)$ n'est pas très séparante.
Comme $\ddc g$ n'est pas modifié lorsqu'on soustrait aux~$f_i$
une même forme linéaire, on peut supposer que l'une d'entre elle est nulle.
Si l'application linéaire résultante $f=(f_1,\dots,f_m)$
n'est pas injective, $g_1$ provient d'une fonction sur un espace
de dimension~$<n$, de sorte que $(\ddc g_1)^n=0$.

Inversement, si la famille $(f_1,\dots,f_m)$ est très séparante,
la fonction $g_1$ est strictement convexe, et sa hessienne est
définie positive en tout point, donc la mesure
$\langle(\ddc g_1)^n,\mu\rangle$ sur~$V$ est le 
produit de la mesure de Lebesgue par une fonction strictement
positive. Cela entraîne que $\int_{W(D)} \langle (\ddc g_1)^n,\mu\rangle>0$
dès que $\dim (W(D))=n$, c'est-à-dire que $\dim(D)=n$.
\end{proof}

\begin{prop}\label{MA-delta_0/2}
Soit $V$ un espace affine réel de dimension~$n\geq1$
muni d'un vecteur-volume~$\mu$.
Soit $(f_1,\dots,f_m)$ une suite non vide de formes affines sur~$V$
et soit $g=\max(f_1,\dots,f_m)$.
Pour $x\in V$, soit $I_x$ l'ensemble des~$i\in\{1,\dots,m\}$
tels que $g(x)=f_i(x)$. Soit $S$ l'ensemble des points~$x$ de~$V$
tels que la famille~$(f_i)_{i\in I_x}$ soit très séparante.

L'ensemble~$S$ est fini
et il existe une famille $(\lambda_x)_{x\in S}$
de nombres réels telle que
\[ (\ddc g)^n =  \sum_{x\in S}\lambda_x\delta_x . \]

En outre, $\lambda_x$ ne dépend que de la famille~$(f_i)_{i\in I_x}$
et du vecteur-volume~$\mu$. De plus, si $\mu$ est strictement positif,
alors $\lambda_x>0$ pour tout~$x\in S$.
\end{prop}
\begin{proof}
Soit $I$ une partie non vide de~$\{1,\dots,m\}$
telle que la famille $(f_i)_{i\in I}$ soit très séparante.
Fixons $i\in I$.
Chaque équation $f_j(x)=f_i(x)$, pour $j\in I$ ($j\neq i$), définit
un sous-espace affine de~$V$; d'après la propriété~\emph b)
du \S\ref{sss-tres-separante}, l'intersection
de ces sous-espaces est de cardinal au plus~$1$.
Par suite, l'ensemble~$S$ est fini. 

On peut supposer que le vecteur-volume~$\mu$ est strictement positif.
Puisque $g$ est convexe,
le supercourant $(\ddc g)^n$ est une mesure positive sur~$V$.
Nous allons prouver que son support est égal à l'ensemble~$S$.

Soit $x$ un point de~$V$. Soit $U$ un voisinage de~$x$
dans~$V$ sur lequel $f_i<g(x)$ si $i\not\in I_x$,
de sorte que $g=\max_{i\in I_x} f_i$ sur~$U$. 
Choisissons d'autre part $j\in I_x$; posons $f'_i=f_i-f_j$
et $g'=\max_{i\in I_x} f'_i$.
Pour tout $y\in U$, on a
\[ g(y)= \max_{i\in I_x} f_i(y) = f_j(y)+ \max_{i\in I_x} (f_i-f_j)(y)
=f_j(y)+ g'(y). \]
Pour tout~$i\in I_x$, posons $f'_i=f_i-f_j$. La famille $(f_i)_{i\in I_x}$
est très séparante si et seulement si la famille $(f'_i)$ l'est,
c'est-à-dire, d'après la propriété \ref{sss-tres-separante}, \emph c),
si et seulement
si l'application affine~$f'\colon V\to \R^{I_x}$ définie par les~$f'_i$
est injective, car $f'_j=0$.

D'autre part, $\ddc f_j=0$, donc $\ddc g=\ddc g'$.
D'après le lemme~\ref{lemm.approx-max},
$g'$ est limite uniforme sur~$U$ d'une famille~$(g'_\eps)$
de fonctions convexes lisses 
provenant de l'espace affine~$V'_x$ image de~$V$ par l'application affine~$f'$.

Si cet espace est de dimension~$<n$, c'est-à-dire si $x\not\in S$,
on a donc $(\ddc g'_\eps)^n=0$ sur~$U$, d'où $(\ddc g')^n=0$ sur~$U$,
puis $(\ddc g)^n=0$ sur~$U$.

Nous avons ainsi démontré que
le support de la mesure  positive $(\ddc g)^n$ est contenu dans
l'ensemble fini~$S$; elle est donc de la forme $\sum_{x\in S}\lambda_x \delta_x$,
où $(\lambda_x)_{x\in S}$ est une famille de nombres réels positifs.

D'autre part, on a vu qu'au voisinage d'un point~$x\in S$,
la fonction~$g$ ne dépend que des formes~$f_i$,
pour $i\in I_x$; par suite, la mesure $(\ddc g)^n$
ne dépend que de ces formes et du vecteur-volume~$\mu$.

La proposition~\ref{MA-delta_0/1} fournit la formule explicite:
\[ \lambda_x = \int_V \langle (\ddc g'_1)^n,\mu\rangle, \]
où $g'_1 = \log \left( \sum_{i\in I_x}  \exp(f'_i)\right)$.
Comme $f'$ est injective, $g'_1$ est strictement convexe
(lemme~\ref{lemm.approx-max}), donc $\langle(\ddc g'_1)^n,\mu\rangle$
est une mesure strictement positive sur~$V$
(lemme~\ref{lemm.MA-H}).
\end{proof}

\begin{rema}
Soit $(f_1,\dots,f_m)$ une famille de formes linéaires
sur~$\R^n$ qui contient un repère affine de~$\R^n$. 
Posons $f=\max(f_1,\dots,f_m)$; c'est une fonction 
linéaire par morceaux (dans la terminologie introduite plus bas,
\emph{paralinéaire}) sur~$\R^n$.
Elle est convexe (comme toute borne supérieure d'une
famille de fonctions convexes);
elle est même strictement convexe au sens des fonctions linéaires par morceaux,
c'est-à-dire que sa restriction à aucune droite vectorielle n'est linéaire.
Si les $f_j$ sont définies sur~$\Q$, la fonction~$f$
est donc la fonction d'appui d'un fibré en droites ample~$L$
sur une variété torique~$X$ de tore~$\gm^n$ (\cite{fulton1993}, p.~70, proposition).
Dans ce cas, l'application~$\nabla(h)$ est appelée {\og application
moment\fg} et l'image de~$V$ est un polytope convexe de~$\R^n$
correspondant à la variété torique polarisée~$(X,L)$.
On sait alors que $\vol(\nabla(h)(V))=\deg_L(X)/n!$.
Le courant $(\ddc f)^n$ est donc égal à $\deg_L(X)\delta_0$.
Cela redonne une formule de Lagerberg (\cite{lagerberg2012}, exemple~3.2
et proposition~5.12).
Le lien entre supercourants et théorie de l'intersection apparaîtra clairement 
dans la suite de ce texte.
\end{rema}

\section{Cellules et polyèdres}

\subsection{}\label{ss.KG}
Soit $K$ un sous-corps de~$\R$ et soit $G$ un sous-$K$-espace vectoriel
non nul de~$\R$.
Dans ce travail, nous appelons espace affine un espace affine
réel~$V$ de dimension finie muni d'un couple $(\vec V_K, V_G)$,
où $\vec V_K$ est une $K$-structure sur l'espace
vectoriel~$\vec V$ sous-jacent à l'espace affine~$V$ 
et $V_G$ est une orbite du sous-groupe $\vec V_K\otimes G$ de~$\vec V$
agissant sur~$V$.

Si $V$ et~$W$ sont de tels espaces affines,
les seules applications applications affines $f\colon V\to W$
que nous considérerons auront une partie linéaire $K$-rationnelle
et enverront~$V_G$ dans~$W_G$.
Concrètement, une application de~$\R^n$ dans~$\R^m$ est affine si elle
est de la forme $x\mapsto A\cdot x+b$, où $A\in\MAT_{m,n}(K)$ et $b\in G^m$.

\subsection{}
Dans la suite de ce travail, qui est motivé par la géométrie analytique,
nous ne considérerons que le cas où $K=\Q$
et $G=\R$ même si la plupart des résultats liés à la géométrie paralinéaire
s'étendraient au cas général. 
En géométrie analytique sur un corps valué~$k$,
on pourrait vouloir garder une trace
des paramètre réels qui interviennent dans les définitions 
des espaces analytiques
(«espaces $\Gamma$-stricts», où $\Gamma$ est
un sous-groupe de~$\R^*_+$ contenant  les rayons 
des briques affinoïdes utilisées, non trivial si $\abs{ k^*}=1$),
ce qui justifierait l'étude du cas $K=\Q$
et $G=\Q\log(\abs{k^\times}\cdot\Gamma)$.
Nous au laissons le lecteur intéressé le soin de modifier les énoncés
et leurs preuves.

Pour éviter toute ambiguïté,
nous qualifierons parfois de fonctions $\R$-paralinéaires
celles qui le sont lorsque $K=\R$.

\subsection{}
Soit $V$ un espace affine et soit $f$ une forme affine sur~$V$, non constante.
Les demi-espaces fermés de~$V$ définis par~$V$ sont les parties~$V_+$
et~$V_-$ définies par les inégalités $f(x)\geq0$ et $f(x)\leq 0$ respectivement.
Les demi-espaces ouverts sont définis de même avec les inégalités strictes.

On appelle \emph{cellule}\index{cellule} de~$V$
toute partie non vide de~$V$ qui est intersection d'une famille finie
de demi-espaces fermés de~$V$. 
En particulier, $V$ est une cellule.

Soit $C$ une cellule de~$V$. On appelle \emph{cellule ouverte}\index{cellule ouverte} associée à~$C$,
et on note~$\mathring C$, l'\emph{intérieur}\index{intérieur d'une cellule}
de la cellule~$C$ dans le sous-espace
affine qu'elle engendre. C'est aussi l'ensemble des points de~$C$
en lesquels $C$ est une variété topologique,
et $C$ est l'adhérence de~$\mathring C$.
La \emph{dimension} de~$C$\index{dimension d'une cellule}
est celle de~$\mathring C$. On note aussi $\partial (C)=C\setminus\mathring C$
le \emph{bord} de la cellule~$C$.\index{bord d'une cellule}

On appelle \emph{polyèdre}  d'un espace affine~$V$ toute partie de~$V$
qui est réunion finie de cellules de~$V$.
En particulier, un polyèdre est fermé.

Il résulte immédiatement de la définition d'une cellule
que l'image réciproque d'une cellule par une application affine est une cellule.
Par suite, l'image réciproque d'un polyèdre est un polyèdre.

Une union finie, une intersection finie de polyèdres sont des polyèdres.

Soit $P$ un polyèdre de~$V$.
On dit qu'une application~$h$ de~$P$ à valeurs dans un espace affine~$W$
est \emph{polyédrale} s'il existe une famille finie $(C_i)$ de cellules de~$V$
recouvrant~$P$ telle que, pour tout~$i$, $h|_{C_i}$ soit une application affine.

\begin{prop}\label{prop.cellule-application-ouverte}
Soit $f\colon V\to W$ une application affine entre espaces affines et
soit $C$ une  cellule de~$V$.

\begin{enumerate}
\item L'image $f(C)$ est une cellule de~$W$.

\item Pour tout point
$x$ de $C$ il existe une section continue $g
\colon f(C)\to C$ de $f$ telle que $g(f(x))=x$ 
et une famille finie $(D_i)$ de cellules de~$W$
telles que $f(C)=\bigcup D_i$ et telles que
$g|_{D_i}$ soit affine pour tout $i$.

\item L'application continue $f|_C\colon C\to f(C)$ est ouverte. 
\end{enumerate}
\end{prop}

Ce découpage de $f(C)$ en cellules est en général, inévitable, même 
si l'on n'impose pas le passage de la section
par un point donné.
Par exemple, lorsque $C$ est un tétraèdre de~$\R^3$, projeté sur un quadrilatère
convexe du plan, de sorte que chaque sommet de~$C$ ait pour image
un sommet de~$f(C)$, il n'existe pas de section affine de~$f$
sur $f(C)$. 

\begin{proof}
Commençons par expliquer comment (c) se déduit de (b). 
Pour démontrer que $f|_C\colon C\to f(C)$ est ouverte, il suffit de
démontrer que pour tout $x\in C$ et tout voisinage $U$ de $x$ dans $C$, l'image
$f(U)$ est un voisinage de $f(x)$ dans $f(C)$. 
Or par (b) il existe une section continue $g$ de $f|_C\colon C\to f(C)$
telle que $g(f(x))=x$. L'image réciproque $g^{-1}(U)$ est alors un voisinage
de $x$ dans $f(C)$ contenu dans $f(U)$, ce qui démontre (c). 

Il reste donc à prouver (a) et (b).
Remarquons d'abord que si $f_1\colon V\to V_1$ et $f_2\colon V_1\to V_2$ vérifient
ces assertions, il en est de même de~$f=f_2\circ f_1$.
Soit en effet $C$ une cellule de~$V$ ; posons $D=f_1(C)$
et $E=f_2(D)=f(C)$. 
Alors $D$ est une cellule de~$V_1$, 
donc $E$ est une cellule de~$V_2$.
D'autre part, soit 
$x$ un point de $C$ ; posons $y=f_1(x))$
et $z=f_2(y)=f(x)$. 
Soit $g_2$ une section continue de $f_2$
sur $E$ à valeurs dans $D$ telle que $g_2(z)=y$ et soit $(E_j)$ une famille finie de cellules de~$V_2$
de réunion~$E$ telle que $g_2|_{E_j}$ soit 
affine pour tout $j$. 
Soit $g_1$ une section continue de $f_1$ sur $D$
à valeurs dans $C$ telle que
$g_1(y)=x$ et soit $(D_i)$
une famille finie de cellules de~$V_2$
de réunion~$D$ telle que $g_1|_{D_i}$ soit 
affine pour tout $i$. 
Alors $g=g_1\circ g_2$ est une section continue de $f$
à valeurs dans $C$ 
telle que $g(z)=x$. 
Pour tout $(i,j)$, posons $F_{ij}=E_j\cap g_2^{-1}(D_i)$ ; 
comme $g_2|_{E_j}$ est affine, c'est une cellule de $W$.
Par construction $E$ est la réunion des $F_{ij}$ et $g|_{F_{ij}}$ est affine
pour tout $(i,j)$.

L'application~$f$ se factorise en l'injection $x\mapsto (x,f(x))$
de~$V$ dans~$V\times W$ et en la projection $(x,y)\mapsto y$ de~$V\times W$
vers~$W$. Compte-tenu de ce qui précède, 
cela permet de supposer que $f$ est soit injective, soit une projection.

Si $f$ est injective, l'énoncé est évident, car $f$ induit
un isomorphisme de~$V$ sur~$f(V)$ donc $f(C)$ est une cellule, 
et $f|_C\colon C\to f(C)$ est un homéomorphisme et en particulier
une application ouverte, ce qui démontre~(a) et~(c). 
En ce qui concerne~(b) la section induite par la réciproque de $f\colon V\to f(V)$ convient
(et est la seule !).

Lorsque $f$ est une projection, il suffit, par récurrence
sur la dimension des fibres, de traiter
le cas de la projection $(x,t)\mapsto x$ de~$V\times\R$ dans~$V$.
Par définition d'une cellule, il existe une famille finie  $(\phi_i)$
de formes affines sur~$V$ et une familles $(a_i)$ 
de nombres rationnels
telles que $C$ soit l'ensemble des~$(x,t)\in V\times\R$ tels que 
$\phi_i(x)-a_it\geq 0$ pour tout~$i$.
Notons $I$, $J$, $K$ les ensembles 
d'indices~$i$ tels que $a_i>0$, $a_i<0$, et $a_i=0$
respectivement.
On a 
\[ f(C) = \bigcap_{(i,j)\in I\times J} \{ \phi_j/a_j \leq \phi_i/a_i \}
\cap \bigcap_{k\in K} \{\phi_k\geq 0\}. \]
C'est donc une cellule de~$V$.

Pour $x$ dans $V$ posons 
\[u(x)=\inf_{i\in I}\phi_i(x)/a_i
\quad\text{et}\quad
v(x)=\sup_{j\in J}\phi_j(x)/a_j\]
(si $I$ est vide $u(x)=+\infty$ et si $J$ est vide $v(x)=-\infty$). 
Les fonctions $u$ et $v$ sont des applications continues de~$V$ 
vers $[-\infty; +\infty]$. Soit $(x,t)$ un point 
de $C$. Pour tout $\xi$ dans $V$, posons
$g(\xi)=\sup(v(\xi), \inf(t,u(\xi))$ ; la fonction $g$ est continue et à valeurs
réelles (même si $u$ ou $v$ prennent des valeurs infinies). 
Si $\xi$ appartient à $f(C)$ alors $v(\xi)\leq u(\xi)$ et donc
$v(\xi)\leq g(\xi)\leq u(\xi)$ ; de plus $\phi_k(\xi)\geq 0$ pour tout
$k\in K$, si bien que $(\xi,g(\xi))$ appartient à~$C$. 
Et l'on par ailleurs $g(x)=t$ car $u(x)\leq t\leq v(x)$ puisque $(x,t)\in C$. 
L'application $h=\xi \mapsto (\xi, g(\xi))$ est une section continue de $f$, 
dont la restriction à $f(C)$ prend ses valeurs dans $C$. 
Pour $i\in I$, notons $D_i$ l'ensemble des points~$\xi$ de~$V$ 
tels que $g(\xi)=\phi_i(\xi)/a_i$ ; pour tout $j\in J$, notons~$E_j$  
l'ensemble des points~$\xi$ de~$V$ 
tels que $g(\xi)=\phi_j(\xi)/a_j$ ; et soit $F$ l'ensemble
des points $\xi$ de $V$ tels que $g(\xi)=t$. 
Les $D_j$, les $E_j$ et $F$ sont définis par des conjonctions finies
d'inégalités larges entre fonctions affines, donc sont des cellules
de $V$, et leur réunion est $V$ tout entier. Par construction, la restriction
de $h$ à chacune d'entre elles est affine. Les traces sur $C$ de ces cellules
répondent aux conditions de~(b). 
\end{proof}

\begin{coro}\label{coro.image-polyedre-sections}
Soit $f\colon V\to W$ une application affine entre espaces affines.
Soit $P$ un polyèdre de~$V$. Alors $f(P)$ est un polyèdre de~$W$.
De plus, il existe une famille finie $(D_i)$ de cellules de~$W$
et une famille finie $(g_i)$ d'applications affines de~$W$ dans~$V$
telles que $f(C)=\bigcup D_i$,
$g_i(D_i)\subset C$ et $f\circ g_i$ soit l'identité
sur~$D_i$.
\end{coro}

\begin{coro}\label{coro.projection-voisinage}
Soit $D$ une cellule de~$\R^m\times \R^d$, 
soit $p\colon\R^n\to\R^m$ la projection sur le premier
facteur et soit $C=p(D)$; posons $d=n-m$.
Soit $(x,t)\in D$ un point qui est dans l'intérieur topologique 
de~$p^{-1}(p(y))\cap D$ dans~$p^{-1}(p(y))$. 

\begin{enumerate}
\item
Pour tout point~$\xi\in C$ assez proche de~$x$,
le point $(\xi,t)$ appartient à~$D$.
\item
La cellule $D$ est un voisinage de~$y$ dans~$C\times\R^d$.
\end{enumerate}
\end{coro}
Si l'assertion~\emph b) est plus générale,
sa démonstration commence par établir l'assertion~\emph a).
\begin{proof}
Pour tout $\xi\in C$, on pose $D_\xi=p^{-1}(\xi)\cap D$;
considérée comme une cellule de~$\R^d$. 

\begin{enumerate}
\item
Par hypothèse, il existe $d+1$ points $t_0,\dots,t_d$ de~$D_x$
affinement indépendants tels que le simplexe~$S$ de sommets $t_0,\dots,t_d$
soit voisinage de~$t$ dans~$\R^d$.
Il existe alors des voisinages ouverts $V_0,\dots,V_d$ de $t_0,\dots,t_d$
tels que pour toute famille $(v_0,\dots,v_d)\in V_0\times \dots\times V_d$,
les points~$v_i$ soient affinement indépendants;
compte tenu de la continuité des coordonnées barycentriques d'un point fixé
en fonction d'un repère affine variable, on peut restreindre les~$V_i$
de sorte que pour toute famille $(v_0,\dots,v_d)$ comme ci-desuss,
le point $t$ appartienne à l'enveloppe convexe des~$v_i$.

Pour tout $i\in\{0,\dots,d\}$, $D\cap (C\times V_i)$ est un voisinage
ouvert de~$(x,t_i)$ dans~$D$,
donc son image~$U_i$ est un voisinage de $p(x,t_i)=x$ dans~$C$,
car $p$ induit une application ouverte de~$D$ vers~$C$
(proposition~\ref{prop.cellule-application-ouverte}, \emph c)).
Soit $V=V_0\cap\dots V_d$; c'est un voisinage de~$x$ dans~$C$.
Soit $\xi\in V$.
Pour tout~$i$, il existe par hypothèse un point $v_i\in V_i$
tel que $(\xi, v_i)\in D$.
Par choix des~$V_i$, le point~$(\xi,t)$ appartient à l'enveloppe
convexe des~$(\xi,v_i)$; comme $D$ est convexe, on a $(\xi, t)\in V$.

\item
Choisissons $d+1$ points affinement indépendants~$t_0,\dots,t_d$
de l'intérieur de~$D_x$ dans~$\R^d$ tel que $t$~appartienne à 
l'intérieur topologique de leur enveloppe convexe~$S$.
Appliquons l'assertion~\emph a) à chacun des points $(x,t_i)$. 
Il existe ainsi, pour tout $i$, un voisinage~$V_i$ de~$x$ dans~$C$
tel que pour tout $\xi\in V_i$, le point $(\xi,t_i)$ appartienne à~$D$.
Alors, $V=\bigcap V_i$ est un voisinage de~$x$ dans~$C$ et pour
tout point de~$V$, le point $(\xi,t_i)$ appartient à~$D$.
Par convexité, $D$ contient~$V\times S$. 
C'est donc un voisinage de~$(x,t)$ dans~$C\times\R^d$.
\qedhere
\end{enumerate}
\end{proof}

\subsection{}
Soit $P$ un polyèdre d'un espace affine~$V$.
On appelle dimension de~$P$ la borne supérieure des dimensions
des cellules contenues dans~$P$.
Soit $x$ un point de~$P$; on appelle dimension de~$P$ en~$x$
la borne inférieure des dimensions des polyèdres~$Q$ contenus
dans~$P$ qui sont un voisinage de~$x$ dans~$P$.
On a $\dim(P)=\sup_{x\in P}\dim_x(P)$.
De plus, $\dim(\emptyset)=-\infty$.

\subsection{}\label{defi.morphisme-pl}
Soit $V$ un espace affine et soit $P$ un polyèdre de~$V$.
Soit $W$ un espace affine.
On dit qu'une application $f\colon P\to W$ est \emph{paralinéaire}\index{paralinéaire!application ---} (\textsc{pl})
s'il existe une famille localement finie $(C_i)$ de cellules de~$V$
dont la réunion est~$P$ et telle que pour tout~$i$, $f|_{C_i}$
est la restriction  d'une application affine de~$V$ dans~$W$.

Une application paralinéaire est continue.

Si $P$ est compact, une fonction paralinéaire de~$P$ dans~$\R$
est exactement une fonction continue définissable (à paramètres dans~$\R$)
dans le langage des groupes abéliens ordonnés
(\cite{marker2002}, chapitre~3).

\subsection{}
Soit $X$ un espace topologique.
Soit $A$ une partie de~$X$  
et soit $(P_i)_{i\in I}$ une famille de parties de~$A$.
On dit que cette famille 
est un \emph{G-recouvrement}\index{G-recouvrement!d'un espace topologique}
de~$A$
si, pour tout point~$a\in A$, il existe une partie \emph{finie}~$J$
de~$I$ telle que $a\in P_j$ pour tout~$j\in J$ 
et telle que $\bigcup_{j\in J} P_j$ est un voisinage de~$a$ dans~$A$.

\subsection{}\label{def.sous-espace-pl-V}
Soit $V$ un espace affine et soit $X$ une partie de~$V$.
On dit que $X$ est un \emph{sous-espace paralinéaire}\index{paralinéaire!sous-espace}
de~$V$ s'il est G-recouvert par les cellules qu'il contient.

Un sous-espace paralinéaire est localement fermé.
Les sous-espaces paralinéaires fermés sont les polyèdres;
tout sous-espace ouvert est un sous-espace paralinéaire.

\subsection{Espace tangent}
Soit $P$ un sous-espace paralinéaire d'un espace affine~$V$.
Soit~$x\in P$. 
Si~$U$ est un voisinage ouvert suffisamment petit
de~$x$ dans~$P$, le sous-espace affine~$\langle U \rangle$
de~$V$ ne dépend que de~$x$, et pas
de~$U$. On l'appelle \emph{l'espace
tangent à~$P$ en~$x$} et on le note $T_xP$. Pour toute cellule~$C$
contenue dans~$P$ et contenant~$x$, on a~$\langle C \rangle \subset T_xP$. 

Si $\mathscr C$ est une décomposition cellulaire de~$P$ l'espace
affine~$T_xP$
est engendré par les cellules de~$\mathscr C$ contenant~$x$. 
En particulier, s'il existe une cellule~$C$ de~$\mathscr C$
telle que $x\in\mathring C$, alors $T_xP=\langle C\rangle$.

\index{espace tangent!en un point d'un sous-ensemble
paralinéaire}

\begin{lemm}
Soit $C$ une cellule de~$\R^n$ et soit $u\colon C\to\R$
une fonction convexe, semi-continue inférieurement.
Si la restriction à~$\mathring C$ de~$u$ est affine,
alors $u$ est affine.
\end{lemm}
\begin{proof}
Soit $\phi$ une fonction affine sur~$\R^n$ qui coïncide avec~$u$
sur~$\mathring C$; prouvons que $\phi$ coïncide avec~$\phi$ sur~$C$.
Soit $x\in\partial(C)$ et soit~$v$ un vecteur non nul de~$\R^n$
tel que $x+\R_+ v$ rencontre~$\mathring C$.
La fonction  $t\mapsto u(x+tv)$ est définie au voisinage
de~$0$ dans~$\R_+$, elle est convexe et semi-continue inférieurement,
donc continue, si bien que $u(x)=\lim \phi(x+tv)=\phi(x)$.
\end{proof}

\begin{prop}\label{prop.continuite-integrale}
Soit $m$, $n$, $d$ des entiers naturels tels que $n=m+d$,
soit $p\colon\R^n\to\R^m$ et $q\colon\R^n\to\R^d$ les projections 
sur les $m$ premières et $d$~dernières coordonnées.
Les fibres de~$p$ sont identifiées à~$\R^d$, muni de la mesure de Lebesgue.

Soit~$D$ une cellule de~$\R^n$, soit $C=p(D)$;
pour tout $x\in C$, on pose $D_x=D\cap p^{-1}(x)$, considérée
comme cellule de l'espace affine~$p^{-1}(x)$.
Soit $\alpha$ une fonction sur~$D$ à support $p$-compact;
pour tout $x\in C$, posons $u(x)=\int_{C_x} \alpha$.
Soit $r$ appartenant à $\N\cup\{+\infty\}$.
Si $\alpha$ est G-localement
de classe $\mathscr C^r$ (resp. G-localement borélienne), 
il en va de même de~$u$.
\end{prop}
\begin{proof}
En considérant des cellules compacte~$C'$ qui G-recouvrent~$C$
et en remplaçant~$D$ par chacune des cellules $D\cap p^{-1}(C')$, 
on se ramène au cas où la cellule~$C$ est compacte.
On remplace alors~$D$ par une sous-cellule compacte d'image~$C$
qui contient le support de~$\alpha$; ainsi, on est ramené à traiter
le cas où $C$ et~$D$ sont compactes.

On traite  le cas $d=1$; le cas général s'ensuit par récurrence.

Pour $x\in C$, posons $f(x)=\inf\{t\;;\; (x,t)\in C\}$
et $g(x)=\sup\{ t\;;\; (x,t)\in C\}$.
Comme on l'a vu dans la démonstration de l'assertion~\emph c)
de la proposition~\ref{prop.cellule-application-ouverte},
les fonctions~$f$ et~$g$ sont paralinéaires (et en particulier continues),
et $D$ est l'ensemble des points~$(x,t)\in \R^m\times\R^d$
tels que $x\in C$ et $f(x)\leq t\leq g(x)$.
On a donc
\[ u(x) = \int_{f(x)}^{g(x)} \alpha(x,t)\,dt \]
pour tout $x\in C$,
et l'assertion en résulte par les propriétés de régularité
des intégrales dépendant de paramètres.
%
\end{proof}

\section{Formes différentielles sur un sous-espace paralinéaire}

\subsection{}\label{sss.def.formes.polyt} 
Soit $V$ un espace affine de dimension finie et
soit $U$ un ouvert de~$V$.
Soit~$\omega\in \mathscr F_{V}^{p,q}(U)$
une forme de type~$(p,q)$ sur~$U$.

Soit $C$ une cellule de~$V$ qui est contenue dans~$U$.
Soit $j$ l'injection canonique de~$\langle C\rangle$ dans~$V$.
Par image inverse, $\omega$ définit une $(p,q)$-forme $j^*\omega$
sur~$j^{-1}(U)$; 
on dit que $\omega|_C=0$ si cette forme~$j^*\omega$ est ponctuellement
nulle en tout point de~$C$ (\S\ref{valeur-ponctuelle}).
Alors $\omega|_D=0$ pour toute cellule~$D$ contenue dans~$C$.

Soit $P$ un sous-espace paralinéaire de~$U$.
On dit que $\omega|_P=0$ si l'on a $\omega|_C=0$ pour toute
cellule~$C$ de~$V$ qui est contenue dans~$P$.  Dans ce cas,
on a $\omega|_Q=0$ pour tout sous-espace paralinéaire
qui est contenu dans~$P$.

\begin{lemm}
\label{lemm.PL-separation-prefaisceau-plonge}
Soit $(P_i)$ un G-recouvrement de~$P$ par des sous-espaces
paralinéaires de~$P$. Si $\omega|_{P_i}=0$
pour tout~$i$, alors $\omega|_P=0$.
\end{lemm}
\begin{proof}
Soit donc~$C$ une cellule de~$V$ qui est contenue dans~$P$;
démontrons que $\omega|_C=0$. 
Alors $(P_i\cap C)$ est un G-recouvrement de~$C$ 
par des sous-espaces paralinéaires
de~$C$; on a $\omega|_{P_i\cap C}=0$ pour tout~$i$.
Il existe alors un G-recouvrement de~$C$ par des cellules~$Q_j$
telles que $\omega|_{Q_j}=0$ pour tout~$j$.
Soit $x$ un point de~$C$. Par définition d'un G-recouvrement,
il existe un ensemble fini~$J$ d'indices~$j$ tel que $x\in \bigcup_{j\in J}
Q_j$ et $\bigcup_{j\in J} Q_j$ soit un voisinage de~$x$ dans~$C$.
Il existe alors~$j\in J$ tel que $\dim(Q_j)=\dim(C)$; par suite,
$\langle Q_j\rangle =\langle C\rangle$. 
Il résulte alors de la définition que la restriction de~$\omega$
à $\langle C\rangle\cap U$ s'annule en tout point de~$Q_j$,
et en particulier de~$x$. Ceci valant pour tout~$x\in C$,
on a $\omega|_C=0$, comme il fallait démontrer.
\end{proof}

\subsection{}
Soit $V$ un espace affine de dimension finie et
soit $P$ un sous-espace paralinéaire de~$V$.
Notons $j$ l'inclusion de~$P$ dans~$V$.

On définit $\mathscr N_P^{p,q}$ comme le sous-faisceau
de~$\mathscr F^{p,q}$ dont les sections sur un ouvert~$U$
sont les formes~$\omega$ telles que $\omega|_{P\cap U}=0$.

Soit $\mathscr G_V$ une sous-$\mathscr A_V$-algèbre de~$\mathscr F_V$.
On définit alors le faisceau~$\mathscr G_P^{p,q}$
des formes de type~$(p,q)$ à coefficients dans~$\mathscr G_V$ sur~$P$ 
par la formule
\[ \mathscr G_P^{p,q} = j^{-1} \left(
\mathscr G_V^{p,q}/(\mathscr  G_V^{p,q}\cap\mathscr N_P^{p,q})\right)
.\]
C'est un sous-faisceau du faisceau~$\mathscr F_P^{p,q}$.

\index{forme lisse!sur un espace paralinéaire}
Une section de ce faisceau est appelée \emph{forme}
de type~$(p,q)$  à coefficients dans~$\mathscr G_V$ 
sur le sous-espace paralinéaire~$P$.
Lorsque la classe~$\mathscr G_V$ est la classe~$\mathscr A_V$
des fonctions lisses, resp. la classe~$\mathscr B_V$ des fonctions
boréliennes, resp. celle des fonctions de classe~$\mathscr C^k$,...
on parlera de \emph{forme lisse}, resp. de forme borélienne,
resp. de forme de classe~$\mathscr C^k$,...

Pour tout voisinage ouvert~$U$ de~$P$  dans~$V$ 
et toute forme~$\alpha\in\mathscr G^{p,q}_V(U)$, 
on notera $\alpha|_P$ 
l'image canonique de~$\alpha$ dans~$\mathscr G^{p,q}_{P}(P)$.
D'après la première partie
du lemme suivant, on a $\alpha|_P=0$ 
si et seulement si $\alpha\in\mathscr N_P^{p,q}(U)$,
c'est-à-dire si $\alpha|_P=0$ au sens introduit au début de ce chapitre.

\begin{lemm}\label{lemm.extensions}
Soit $V$ un espace affine de dimension finie,
soit $P$ un sous-espace paralinéaire de~$V$.

Pour tout ouvert~$U$ de~$V$, l'application canonique
\[ \mathscr G_V^{p,q}(U)/ \left( \mathscr G_V^{p,q}(U)\cap \mathscr N_P^{p,q}(U)\right) \to \mathscr G_P^{p,q}(U\cap P) \]
est injective.

Si de plus $U\cap P$ est fermé dans~$U$, cette application est bijective.
\end{lemm}
\begin{proof}
Démontrons que l'application donnée est injective.
Soit $\omega\in\mathscr G_V^{p,q}(U)$ d'image nulle
dans~$\mathscr G_P^{p,q}(U\cap P)$ et démontrons
que $\omega\in\mathscr N_P^{p,q}(U)$.
Il existe donc une famille~$(U_i)$ d'ouverts de~$U$
telle que $(U_i\cap P)$ soit un recouvrement de~$U\cap P$
et telle que $\omega|_{U_i}\in\mathscr N_P^{p,q}(U_i)$
pour tout~$i$.
On a alors $\omega|_{P\cap U_i}=0$ pour tout~$i$,
d'où $\omega|_P=0$ (lemme~\ref{lemm.PL-separation-prefaisceau-plonge}).

Supposons maintenant que $U\cap P$ est fermé dans~$U$
et démontrons que cette application est surjective.
Commençons par observer que 
$\mathscr G_P^{p,q}\cap\mathscr N_P^{p,q}$
est un $\mathscr A_V$-module, donc est fin. Par suite, sa
cohomologie est triviale sur tout ouvert, si bien que 
l'application canonique
\[ \mathscr G_V^{p,q}(U) \to (\mathscr G_V^{p,q}/\mathscr G_V^{p,q}\cap\mathscr N_P^{p,q})(U) \]
est surjective
pour tout ouvert~$U$ de~$V$.

Soit $\omega\in\mathscr G_P^{p,q}(U\cap P)$.
Par définition et par ce qui précède, $\omega$ est donnée
par une famille $(U_i)$ d'ouverts de~$U$
recouvrant~$U\cap P$ et, pour tout~$i$, un élément
$\omega_i$ de $\mathscr G_V^{p,q}(U_i)$
sujets à la condition de compatibilité suivante: pour tout couple~$(i,j)$,
il existe une famille $(U_{i,j}^k)_k$ d'ouverts de~$U_i\cap U_j$
recouvrant $U_i\cap U_j\cap P$ telle que 
$\omega_i|_{U_{i,j}^k}-\omega_j|_{U_{i,j}^k}\in\mathscr N_P^{p,q}(U_{i,j}^k)$ pour tout~$k$.
Posons $V_{i,j}=\bigcup_k U_{i,j}^k$; 
on a 
$\omega_i|_{V_{i,j}}-\omega_j|_{V_{i,j}}\in \mathscr N_P^{p,q}(V_{i,j})$.
Puisque $V_{i,j}$ est un voisinage de~$U_i\cap U_j\cap P$ 
dans~$U_i\cap U_j$, 
on a donc $\omega_i|_{U_i\cap U_j}-\omega_j|_{U_i\cap U_j}\in\mathscr N_P^{p,q}(U_i\cap U_j)$.
Comme $P$ est fermé dans~$U$, $U_0=U\setminus P$ est un ouvert de~$U$;
Adjoignons-le à la famille~$(U_i)$ 
et considérons une partition de l'unité $(\phi_i)$ sur~$U$
formée de fonctions lisses et
subordonnée au recouvrement~$(U_i)$ de~$U$ ainsi obtenu.
Le support de~$\phi_0$ est disjoint de~$P$, donc $\phi_0$ est nulle au voisinage de~$P$.
La fonction $\phi=\sum_{i\neq 0}\phi_i=1-\phi_0$ est donc
égale à~$1$ au voisinage de~$P$.
Pour tout~$i\neq0$, le support de la section~$\phi_i\omega_i$ est contenu dans~$U_i$
et est contenu dans le support de~$\phi_i$ qui est fermé dans~$U$,
de sorte qu'il est lui-même fermé dans~$U$;
notons $(\phi_i\omega_i)^U$ le prolongement (par zéro) à tout l'espace
de cette section.
Posons $\omega'=\sum_{i\neq0} (\phi_i\omega_i)^U$.
Pour prouver que $\omega'|_P=\omega$,
il suffit de vérifier que pour tout~$i\neq0$, 
$\omega'|_{U_i}$ coïncide avec~$\omega_i$ modulo~$\mathscr N_P^{p,q}$.
Fixons un indice~$i\neq0$.
La famille~$(\phi_j|_{U_i})_j$ est une partition
de l'unité sur~$U_i$ subordonnée au recouvrement~$(U_i\cap U_j)_j$, et l'on a
\begin{align*}
\omega'|_{U_i} 
&= \sum_{j\neq0} (\phi_j \omega_j)^U|_{U_i} \\ 
&= \sum_{j\neq0} (\phi_j \omega_j|_{U_i\cap U_j})^{U_i} \\ 
& = \sum_{j\neq0} (\phi_j (\omega_j-\omega_i)|_{U_i\cap U_j})^{U_i} 
+\sum_{j\neq0} \phi_j\omega_i |_{U_i}\\
&  =\sum_{j\neq0} (\phi_j (\omega_j-\omega_i)|_{U_i\cap U_j})^{U_i} 
-\phi_0 \omega_i +\omega_i
\end{align*}
Puisque $\omega_j-\omega_i\in\mathscr N_{P}^{p,q}(U_i\cap U_j)$,
la forme $(\phi_j (\omega_j-\omega_i))^{U_i}$
appartient à $\mathscr N_P^{p,q}(U_i)$.
Comme $\phi_0$ est nulle au voisinage de~$P$,
le terme $-\phi_0\omega_i$ appartient à~$\mathscr N_P^{p,q}(U_i)$.
Cela prouve que $\omega'|_{U_i}$ coïncide avec~$\omega_i$ au voisinage de~$P$
modulo~$\mathscr N_P^{p,q}$
et conclut la démonstration.
\end{proof}

\subsection{}
Soit $P$ un sous-espace paralinéaire de~$V$.
Notons $\mathscr G_{P}$ le faisceau $\mathscr G^{0,0}_P$.
Par construction, $\mathscr G_P$ est le sous-faisceau du
faisceau des fonctions numériques sur~$P$ qui sont 
localement restriction d'une fonction de~$\mathscr G_V$ au voisinage de~$P$.
En vertu du lemme~\ref{lemm.extensions},
ses sections sont même globalement restriction d'une telle
fonction.
Elles seront appelées \emph{fonctions de classe~$\mathscr G_V$} sur~$P$,
et \emph{fonctions lisses} lorsque $\mathscr G_V=\mathscr A_V$.

\subsection{}
Une forme nulle sur un sous-espace paralinéaire
étant nulle sur tout sous-espace, on dispose pour tout couple~$(P,Q)$
de sous-espaces paralinéaires de~$V$ tel que $Q\subset P$,
d'une application de restriction
de~$\mathscr G_P^{p,q}(P)$  dans~$\mathscr G_Q^{p,q}(Q)$.
On obtient ainsi un préfaisceau sur~$V$ pour la G-topologie,
que l'on note encore, par abus, 
$\mathscr G_{V}^{p,q}$. C'est un sous-G-préfaisceau
du G-préfaisceau~$\mathscr F_{V}^{p,q}$.

Soit~$P$ un sous-espace paralinéaire de~$V$.
Observons que si $Q$ est un ouvert de~$P$,
on a $\mathscr G_Q^{p,q}(Q)=\mathscr G_P^{p,q}(Q)$;
plus généralement, le faisceau~$\mathscr G_Q^{p,q}$ est la restriction à~$Q$
du faisceau~$\mathscr G_P^{p,q}$.
Autrement dit, la restriction à la catégorie des ouvert de~$P$
du G-préfaisceau $\mathscr G_{V}^{p,q}$.
est le faisceau~$\mathscr G_P^{p,q}$.
Cela justifie de noter encore
$\mathscr G_{P}^{p,q}$ la restriction à~$P$
du G-préfaisceau~$\mathscr G_{V}^{p,q}$.

La somme directe $\mathscr F^{*,*}_{V}$ est muni
d'une structure naturelle de $\mathscr A_V$-algèbre graduée
et $\mathscr G_{V}^{*,*}$ en est une sous-algèbre
graduée;
c'est en particulier le cas si $\mathscr G_V=\mathscr A_V$.

\subsection{}
L'opérateur~$\mathrm J$ passe au quotient par~$\mathscr N_P^{*,*}$,
pour tout sous-espace paralinéaire~$P$ de~$V$.
Il induit donc un opérateur~$\mathrm J$ sur
le G-préfaisceau~$\mathscr F^{*,*}_{V}$.
C'est un morphisme d'algèbres.
On dit encore qu'une forme~$\omega$ est symétrique si elle vérifie
$\mathrm J\omega=\omega$.

\subsection{}
Soit $\mathscr G_V$ et $\mathscr H_V$ des sous-faisceaux 
en $\R$-vectoriels de~$\mathscr F_V$
tels que les dérivées partielles de toute fonction de~$\mathscr G_V$
soient définies et appartiennent à~$\mathscr H_V$.
Alors, les opérateurs~$\di$ et~$\dc$ 
induisent deux opérateurs différentiels~$\di$ et~$\dc$
de $\mathscr G_{V}^{*,*}$ vers~$\mathscr H^{*,*}_{V}$.
Si $\mathscr G_V\subset \mathscr C^2_V$, ces opérateurs vérifient
$\di\di=\dc\dc=0$ et $\di\dc=-\dc\di$.

Lorsque $\mathscr G_V=\mathscr H_V=\mathscr A_V$ sont le faisceau des fonctions lisses sur~$V$, les opérateurs $\di$ et~$\dc$
sont des dérivations au sens des algèbres graduées.

Les formules du lemme~\ref{lemm.d'-d''-J} restent valables dans ce contexte.

\begin{lemm}\label{lemm.formes-nulles-local}
Si $P$ est un sous-espace paralinéaire de~$V$ tel que $\dim(P)<\max(p,q)$,
le G-préfaisceau $\mathscr F^{p,q}_P$ est nul.
\end{lemm}
\begin{proof}
Soit $Q$ un sous-espace paralinéaire de $P$ et 
soit $\alpha$ un élément de~$\mathscr  F^{p,q}_P(Q)=\mathscr F^{p,q}_Q(Q)$;
prouvons que $\alpha=0$.
Comme on peut raisonner localement sur~$Q$, on peut supposer que $\alpha$
est donné par une forme $\omega\in\mathscr F^{p,q}_V(W)$,
où $W$ est un voisinage ouvert de~$Q$ dans~$V$.
Pour toute cellule~$C$ de~$V$ qui est contenue dans~$Q$,
one a $\dim(C)\leq \dim(P)<\max(p,q)$,
et donc $\omega|_C=0$.
Cela démontre que $\omega|_{Q}=0$, d'où le lemme.
\end{proof}

\subsection{}
Soit $V$ et~$V'$ des espaces affines et soit $\phi\colon V\to V'$
une application affine.
 
Soit $P$ et~$P'$ des sous-espaces paralinéaires 
de~$V$ et~$V'$ respectivement  tels que $\phi(P)\subset P'$.
L'image par~$\phi$ d'une cellule  de~$V$ étant une cellule de~$V'$,
on en déduit que si une forme~$\alpha$ de type~$(p,q)$ au voisinage de~$P'$
est nulle  en restriction à~$P'$, 
son image réciproque (qui est définie au voisinage
de~$P$) est nulle en restriction à~$P$.
Il en résulte une application linéaire $\phi^*\colon \mathscr F^{p,q}_{P'}(P')
\to \mathscr F^{p,q}_P(P)$. On construit ainsi
un morphisme de G-préfaisceaux
$\phi^*\colon  \mathscr F^{p,q}_{V} \to
\phi_*\mathscr F^{p,q}_{V'}$ sur~$V$.

L'homomorphisme~$\phi^*$ respecte 
les structures de modules ou d'algèbres graduées,
commute à l'involution~$\mathrm J$,
et aux opérateurs différentiels~$\di$ et~$\dc$ lorsque
ceux-ci sont définis.

On a $\id^*=\id$ et $(\phi\circ \psi)^*=\psi^*\circ \phi^*$.

Soit $\mathscr G_V$ une sous-$\mathscr A_V$-algèbre de
de~$\mathscr F_V$.
Si une forme $\omega$ est à coefficients dans~$\mathscr G_V$, $\phi^*\omega$
est à coefficients dans la $\mathscr A_{V'}$-algèbre~$\mathscr G_{V'}$
engendrée par les fonctions sur~$V'$ 
qui sont localement de la forme $u\circ\phi$, 
où $u$ est une section de~$\mathscr G_V$. En particulier, si $\omega$
est lisse, resp. de classe~$\mathscr C^k$, resp. borélienne,
il en est de même de~$\phi^*\omega$.

Soit $f\colon P\to P'$ l'application paralinéaire déduite
de~$\phi$ par restriction.
Le morphisme~$\phi^*$ induit un  morphisme de G-préfaisceaux sur~$P$
de~$\mathscr F^{p,q}_{P}$ dans~$f_*\mathscr F^{p,q}_{P'}$.
On vérifie qu'il ne dépend que de~$f$, et pas de~$\phi$;
on le note naturellement~$f^*$.

 
Si $P$ et $Q$ sont deux
sous-ensembles linéaires par
morceaux de $V$ tels que $Q\subset P$
et si $j$ désigne l'inclusion de~$Q$ dans~$P$, observons que l'on a
$j^*\alpha=\alpha|_Q$.

\begin{rema}\label{rema.G-loc-inj}
Si l'application~$f$ possède G-localement une section,
alors l'homomorphisme $f^*\colon \mathscr F_{P'}^{p,q}(P')\to\mathscr F_P^{p,q}(P)$ 
est injectif.
C'est notamment le cas lorsque
l'application~$f$ est compacte et surjective, 
d'après la proposition~\ref{prop.sections-G-locales}.
\end{rema}
%

\begin{exem}\label{exem.formes-cellules-valeurs}
Soit $P$ un sous-espace paralinéaire d'un espace affine~$V$.
Soit $W=\langle P\rangle$ le sous-espace affine engendré par~$P$ 
et soit $x_1,\dots,x_m$ des coordonnées sur~$W$.
Les formes 
\[ \di x_{i_1}\wedge\dots \wedge \di x_{i_p}
\wedge \dc x_{j_1}\wedge\dots\wedge \dc x_{j_q} ,\]
où $(i_1,\dots,i_p)$ et $(j_1,\dots,j_q)$ parcourent 
les suites strictement croissantes d'éléments
de $\{1,\dots,m\}$ de longueurs~$p$ et~$q$ respectivement,
forment une famille génératrice 
de $\mathscr F^{p,q}_P(P)$ comme module sur
$\mathscr F_P(P)$.

Pour qu'un élément de $\mathscr F^{p,q}_P(P)$
appartienne à~$\mathscr G^{p,q}_P(P)$, il suffit qu'on puisse
l'écrire comme une combinaison linéaire de ces formes à coeffients
dans~$\mathscr G_P(P)$.

Supposons de plus que $P$ soit une cellule.
Alors, les formes ci-dessus constituent une base de $\mathscr F^{p,q}_P(P)$,
c'est-à-dire que $\mathscr F^{p,q}_P(P)$ s'identifie 
à $\mathscr F_P(P) \otimes_\R A^{p,q}(\vec W)$.
L'évaluation en un point~$x\in P$ fournit alors un
morphisme d'algèbres graduées
$\alpha\mapsto \alpha(x)$ de~$\mathscr F^{*,*}_P(P)$
dans~$A^{*,*}(\vec W)$.
Pour que l'on ait $\alpha=0$, il faut et il suffit
que l'on ait $\alpha(x)=0$ pour tout $x\in P$.
On dira ainsi que la nullité de~$\alpha$ peut se vérifier \emph{ponctuellement}.

Soit $P'$ une cellule d'un espace affine~$V'$ et soit $f\colon V'\to V$
une application affine telle que $f(P')\subset P$. 
Soit $W'$ l'espace affine engendré par~$P'$;
l'application~$f$ induit un homomorphisme de~$W'$ dans~$W$,
et un homomorphisme linéaire $f^*$ de~$\vec W^*$ dans~$\vec W^{\prime*}$,
d'où un homomorphisme linéaire, toujours noté~$f^*$,
de~$A^{p,q}(\vec W)$ dans~$A^{p,q}(\vec W')$.
Pour toute forme~$\alpha\in \mathscr F^{p,q}_P(P)$
et tout point $y\in P'$, on a $(f^*\alpha)(y)=f^*(\alpha (f(y)))$.
\end{exem}

\begin{exem}\label{exem.forme-pl-produit}
Soit $P$ un sous-espace paralinéaire d'un espace affine
de la forme~$V\times W$.  Notons $p_1$ et $p_2$ les deux projections
de~$V\times W$ sur~$V$ et~$W$, 
choisissons des coordonnées~$(x_i)$ sur~$V$ et~$(y_j)$ sur~$W$.
D'après l'exemple précédent, les formes du type $\di x_i$, $\di y_j$,
$\dc x_i$ et $\dc y_j$ engendrent~$\mathscr F^*_P(P)$ comme algèbre
sur $\mathscr F_P(P)$.
On en déduit que les formes sur~$P$ du type $p_1^*\alpha \wedge p_2^*\beta$,
où $\alpha$ et $\beta$ sont des formes sur~$V$ et~$W$,
engendrent~$\mathscr F^*_P(P)$ comme module sur $\mathscr F_P(P)$.
\end{exem}

\begin{rema}\label{rema.prefaisceau-plonge}
Observons que le G-préfaisceau $\mathscr F^{p,q}_{{V}}$  est séparé.

Soit en effet un sous-espace paralinéaire~$P$ de~$V$
et soit $(P_i)$ un G-recouvrement de~$P$ par des sous-espaces paralinéaires.
Soit $\omega\in\mathscr F^{p,q}_P(P)$
et supposons que $\omega|_{P_i}=0$ pour tout~$i$; démontrons
que $\omega=0$.
Comme $\mathscr F^{p,q}_P$ est un faisceau,
il suffit de prouver que $\omega$ est nulle localement,
ce qui permet de supposer que~$\omega$
provient d'une forme~$\alpha$ définie sur un voisinage~$U$ de~$P$.
Dans ce cas, l'assertion a été démontrée au \S\ref{sss.def.formes.polyt}.

Cependant, ce G-préfaisceau n'est en général 
pas un faisceau pour la G-topologie de~$P$,
comme on le constate déjà
lorsque $\mathscr G_V=\mathscr A_V$, $\dim(V)>0$ et $P=V$:
une fonction $\mathscr C^\infty$ par morceaux sur~$V$
n'est pas $\mathscr C^\infty$ en général.
\end{rema}

\subsection{}\label{equerre}
On notera que la notion de fonction (et donc de forme) lisse 
sur un sous-ensemble paralinéaire d'un espace affine dépend de la façon
dont ce sous-ensemble est plongé.
\begin{figure}[htb]
\centering
\begin{tikzpicture}[scale=1]
\draw (-1.5,0) -- (1.5,0);
\draw (0,-0.5) -- (0,1.5) ;
\draw[very thick] (-1,0)--(0,0)--(1,1);
\draw (0.5,0.5) node[above] {$P$};
\end{tikzpicture}
\hfil
\begin{tikzpicture}[scale=1]
\draw (-1.5,0) -- (1.5,0);
\draw (0,-0.5) -- (0,1.5) ;
\draw[very thick] (-1,0)--(0,0)--(1,0);
\draw (0.5,0) node[above] {$Q$};
\end{tikzpicture}
\par\end{figure}
Soit par exemple~$S$
le segment~$[(-1,0), (0,0)]$
de~$\R^2$, soit~$P$ la
réunion
de~$S$ et du segment~$[(0,0),(1,1)]$,
et soit~$Q=[-1,1]$ l'image
de~$P$ dans~$\R$ par la première projection~$p\colon\R^2\ra\R$.
L'application~$p$ induit un isomorphisme d'espaces paralinéaires
de~$P$ sur~$Q$ (c'est-à-dire un homéomorphisme 
modulo lequel les fonctions paralinéaires sur la source et le but sont les mêmes). 

La fonction~$f$ sur~$Q$ définie par $f(x)=\max(x,0)$ est continue,
mais n'est pas lisse. En revanche, la fonction $f\circ p$
sur~$P$ est lisse, comme restriction de la seconde coordonnée.

%

\begin{prop}\label{prop.pl-lisse-polytope}
Soit $P$ un sous-espace paralinéaire de~$V$ et soit $f$ une fonction
sur~$P$ qui est simultanément lisse et paralinéaire. 
Pour tout point~$x$ de~$P$, 
il existe une forme affine~$g$ sur~$V$ telle que $f$ et~$g|_P$
coïncident au voisinage de~$x$.
\end{prop}
\begin{proof}
Soit $\mathscr C$ une famille finie de cellules de~$V$ contenant~$x$
dont la réunion~$Q$ est un voisinage de~$x$ 
et telles que pour tout $C\in\mathscr C$,
la restriction de~$f$ à~$C$ soit affine. 
Nous allons démontrer qu'il existe une forme affine~$g$ sur~$V$
telle que $g|_Q=f|_Q$.
Quitte à remplacer~$V$ par son sous-espace affine engendré par~$Q$,
on suppose que $\langle Q\rangle=V$; pour simplifier les notations,
on suppose aussi que $V$ est un espace vectoriel, que $x=0$ et que $f(x)=0$.
Soit $\phi$ une fonction~$\mathscr C^\infty$ sur~$V$ au voisinage de~$Q$
telle que $\phi|_Q=f|_Q$ et soit $g$ la différentielle de~$\phi$ en l'origine.
Pour toute cellule~$C\in\mathscr C$, la fonction $f|_C$ coïncide avec~$g|_C$
car $f|_C$ est linéaire. On a donc $f|_Q=g|_Q$.

L'application $g\colon V\to\R$ est linéaire, mais elle n'est pas 
a priori rationnelle pour  la $\Q$-structure implicitement considérée sur~$V$.
Démontrons que c'est pourtant le cas. Pour $C\in\mathscr C$,
le sous-espace vectoriel~$V_C$ engendré par~$C$ est $\Q$-rationnel,
par définition des sous-espaces paralinéaires et des cellules que l'on considère,
de même que $f|_C$, considérée comme une forme linéaire~$f_C$ sur~$V_C$.
Comme la somme des espaces~$V_C$ est égale à~$V$, on a par dualité une
injection linéaire $\iota\colon V^\vee \to \bigoplus V_C^\vee$.  La démonstration prouve
que $g$ est l'antécédent de la famille $(f_C)$ dans~$V^\vee$.
Comme les $f_C$ sont $\Q$-rationnelles, de même l'application~$\iota$,
la forme~$g$ est $\Q$-rationnelle.
\end{proof}

\section{G-formes positives sur un sous-espace paralinéaire}

\begin{defi}
Soit $V$ un espace affine de dimension finie sur~$\R$
et soit $P$ un sous-espace paralinéaire de~$V$.
Soit $\omega$ une
forme  de type~$(p,p)$ sur~$P$.
On dit que $\omega$  est \emph{positive} (\resp \emph{faiblement positive},
\resp \emph{fortement positive}) 
si pour toute cellule~$C$ contenue dans~$P$
et tout point~$x\in C$, 
l'élément~$\omega|_C(x)$ 
de~$A^{p,p}(\overrightarrow{\langle C\rangle})$
est une $(p,p)$-forme positive (\resp faiblement positive,
\resp fortement positive) 
sur l'espace vectoriel~$\overrightarrow{\langle C\rangle}$
au sens de la définition~\ref{defi.positive-gt}.
\end{defi}

Compte tenu du lemme~\ref{lemme.image.inverse}
et de la compatibilité des morphismes d'évaluation à la restriction
à une sous-cellule, il suffit
de vérifier cette condition pour toute cellule maximale
d'une décomposition cellulaire de~$P$.

En particulier, ce sont des notions G-locales:
soit $\alpha\in\mathscr F^{p,p}_P(P)$ et soit $(P_i)$ un G-recouvrement
de~$P$, \emph{pour que $\alpha$ soit positive, il faut et il suffit
que $\alpha|_{P_i}$ soit positive pour tout~$i$.}
Comme $P$ est paracompact (c'est un sous-espace de~$\R^n$!),
l'assertion se déduit
de la proposition~\ref{prop.decomp-cell}
qui affirme l'existence d'une décomposition cellulaire 
adaptée au  G-recouvrement~$(P_i)$.
Les cas des formes faiblement positives et fortement positives
sont analogues.

\begin{prop}\label{prop.diff-positive}
Soit $P$ un sous-espace paralinéaire d'un espace affine~$V$.
\begin{enumerate}
\item
Pour toute forme lisse~$\alpha$, symétrique de type~$(p,p)$ sur~$P$,
il existe des formes lisses positives
$\alpha_1$ et~$\alpha_2$
telles que $\alpha=\alpha_1-\alpha_2$.
\item Supposons que $P$ soit compact
et soit $u$ une fonction lisse sur~$P$. Il existe
des fonctions lisses~$u_1$ et~$u_2$ sur~$P$
telles que $u=u_1-u_2$ et telles que les formes $\ddc u_1$
et $\ddc u_2$ soient positives.
\end{enumerate}
\end{prop}
\begin{proof}
\begin{enumerate}
\item
On peut supposer que $V=\R^n$.
Soit $\lambda$ la $(1,1)$-forme $\ddc\norm{x}^2$ sur~$\R^n$.
On a 
\[ \lambda = 2 \sum_{i=1}^n \di x_i \wedge \dc x_i \ ;\]
elle est donc symétrique et définie positive au sens
du paragraphe~\ref{subsec.formes-positives}.
D'après la proposition~\ref{prop.harvey},
la forme $\lambda^p$ appartient à l'intérieur 
du cône des formes de type~$(p,p)$ 
positives sur~$\R^n$.
Par  suite, il existe pour tout point~$x$ de~$P$ 
un voisinage ouvert~$U_x$ 
de~$x$ dans~$P$ et un nombre réel strictement positif~$c_x$
tel que $\alpha+c_x\lambda^p$ soit une forme positive
sur~$U_x$.

Soit $U$ une partie ouverte de~$\R^n$ telle que $P$ soit un sous-espace
fermé de~$U$. Il est classique qu'il existe une fonction $\psi\colon U\to\R_+$, 
de classe~$\mathscr C^\infty$ et propre. (Cela se voit en considérant
un recouvrement localement fini de~$U$ par des boules ouvertes
relativement compactes dans~$U$, en prenant l'inverse 
de la somme de fonctions-plateaux supportées par ces boules,
en ajoutant $\norm x^2$.)
Par compacité, il existe alors pour tout entier~$k$
un nombre réel~$c_k>0$ tel que 
pour tout $x\in P\cap \psi^{-1}([0,k])$, 
il existe $y\in P$ tel que $x\in U_y$
et $c_k\geq c_y$. Soit alors $\phi\colon\R_+\to\R$
une fonction de classe~$\mathscr C^\infty$ 
telle que $\phi(t)\geq c_k$
si $t\leq k$. Pour tout $x\in P$, la forme $\alpha(x)+\phi(\psi(x))\lambda(x)^p$
sur~$\R^n$ est positive. Il suffit donc
de poser $\alpha_1=\alpha+ (\phi\circ \phi)\lambda^p$
et $\alpha_2 = (\phi\circ\phi)\lambda^p$.


\item
Posons $p=1$ et $\alpha=\ddc u$.
Comme on a supposé $P$ compact, les arguments
de la première partie prouvent qu'il existe un nombre réel~$c>0$
tel que  la forme $\alpha+c \lambda$ soit fortement positive sur~$P$.
Il suffit alors de $u_1=u+c \norm x^2$
et $u_2=c \norm x^2$, d'où la proposition.
\qedhere
\end{enumerate}
\end{proof}

\begin{prop}\label{prop.produit-formes-positives-pl}
Soit $P$ un sous-espace paralinéaire d'un espace affine~$V$
et soit $\omega_1,\dots,\omega_s$ des formes  sur~$P$;
posons $\omega=\omega_1\wedge\dots\wedge\omega_r$.

\begin{enumerate}
\item
Si les $\omega_j$
sont toutes fortement positives, alors $\omega$
est fortement positive.

\item
Si les $\omega_j$ sont toutes positives, alors $\omega$ est positive.

\item 
Si les $\omega_j$ sont toutes faiblement positives 
et si au plus une d'entre elle n'est pas fortement positive,
alors $\omega$ est faiblement positive.
\end{enumerate}
\end{prop}
\begin{proof}
Par hypothèse, les formes~$\omega_j$ sont 
restrictions de formes au voisinage,de~$P$.
L'assertion se déduit donc cellule par cellule
de la proposition~\ref{prop.produit-positives}.
\end{proof}

\subsection{}
Soit $P$ un sous-espace paralinéaire d'un espace
affine~$V$ et soit $n=\dim(P)$. 
Les notions de formes faiblement positives, positives
et fortement positives coïncident en bidegrés~$(0,0)$ et $(n,n)$;
pour les formes symétriques, 
elles coïncident en bidegrés~$(1,1)$ et  $(n-1,n-1)$.

Cela se démontre par inspection  cellule par cellule.


\def\CConv{\operatorname{CConv}}

\begin{defi}\label{defi.convexe-pl}
Soit $V$ un espace affine  de dimension finie et
soit $P$ une partie paralinéaire de~$V$.
On définit le faisceau~$\CConv_P$ comme le plus petit sous-faisceau
du faisceau~$\mathscr C_P$ des fonctions continues sur~$P$
qui vérifie les deux propriétés suivantes:
\begin{enumerate}
\item Il contient toute fonction lisse~$u$ sur un ouvert de~$P$
telle que la forme~$\ddc u$ soit positive;
\item Il est stable par limite uniforme.
\end{enumerate}
\end{defi}
Ses sections sont appelées fonctions convexes continues sur~$P$.

\begin{lemm}\label{lemm.convexe-cone}
Soit $V$ un espace affine  de dimension finie et
soit $P$ une partie paralinéaire de~$V$.
Pour tout ouvert~$U$ de~$P$, l'ensemble $\CConv_P(U)$
des fonctions convexes continues sur~$U$ est un sous-cône convexe
de~$\mathscr C(U)$.
\end{lemm}
\begin{proof}
Les fonctions continues~$u$ telles que $\lambda u$ soit convexe, 
pour tout nombre réel positif~$\lambda$, forment un sous-faisceau
de~$\CConv_P$ qui vérifie les conditions~(i) et (ii),
donc est égal à~$\CConv_P$. 
Cela prouve que $\CConv_P(U)$ est un cône
pour tout ouvert~$U$ de~$X$.

Démontrons que si $u$ et $v$ sont des fonctions convexes sur un ouvert~$U$
de~$P$, il en est de même de $u+v$.

Supposons d'abord que $v$ soit lisse.
Les fonctions continues~$f$ telles que $f+v$ soit convexe 
forment un sous-faisceau
de~$\CConv_U$  qui contient les fonctions lisses et est stable
par limite uniforme; il contient donc toutes les fonctions convexes sur~$U$.

Par commutativité de l'addition,
$u+v$ est convexe dès que $u$ ou~$v$ est lisse.

Dans le cas général, les fonctions continues~$f$ telles que $f+v$ 
soit convexe forme un sous-faisceau de~$\CConv_U$ qui contient les fonctions 
lisses convexes d'après le premier cas traité et est stable par limite uniforme;
il contient donc toutes les fonctions convexes sur~$U$;
par suite, $u+v$ est convexe.
\end{proof}

\begin{lemm}
\begin{enumerate}
\item
Si $P$ est une cellule, une fonction continue sur~$P$ est convexe
si et seulement si elle est convexe au sens usuel.
\item
Soit $f\colon W\to V$ une application affine et soit $Q$ un sous-espace
paralinéaire de~$W$ tel que $f(Q)\subset P$.
Pour toute fonction convexe~$u$ sur~$P$, la fonction $u\circ f$ est convexe sur~$Q$.
\end{enumerate}
\end{lemm}
\begin{proof}
\begin{enumerate}
\item
Pour éviter les confusions, nous appellerons ici \pl-convexes
les fonctions convexes au sens de la définition~\ref{defi.convexe-pl},
et simplement convexes celles qui le sont au sens usuel.

Les fonctions de $\CConv_P$ formé des fonctions qui sont convexes
forment un sous-faisceau de~$\CConv_P$.
En effet la convexité est une propriété locale, comme on le voit par 
par réduction à la dimension~1 via le critère de croissance des dérivées
à droite et à gauche.
Ce sous-faisceau contient les fonctions lisses~$u$ telles que $\ddc u\geq 0$;
on retrouve ici le critère hessien de convexité des fonctions~$\mathscr C^2$.
Il est stable par limite uniforme, par la définition directe.
Par suite, ce sous-faisceau coïncide avec~$\CConv_P$:
cela prouve que toute fonction \pl-convexe sur~$P$ est convexe.
 
Inversement, soit $u$ une fonction convexe continue sur~$P$;
nous allons démontrer qu'elle est localement limite uniforme de fonctions
convexes lisses sur~$P$; cela permet de supposer que $P$ est compact
et  nous allons alors montrer qu'elle est limite uniforme de
fonctions convexes lisses sur~$P$. On peut également supposer
que $P$ est d'intérieur non vide;
fixons une origine~$o$ dans l'intérieur de~$P$.
Lorsque $\lambda<1$ tend vers~$1$, les fonctions $u_\lambda\colon
x\mapsto u(\lambda x)$ sont définies au voisinage de~$P$
et convergent uniformément vers~$u$ sur~$P$, car $u$  est uniformément
continue. Cela permet de supposer que $u$ se prolonge au voisinage de~$P$
en une fonction qui est encore continue et convexe ; nous noterons~$u$
un tel prolongement.
Pour tout noyau positif régularisant~$(\rho_\eps)$ sur~$\R^n$,
les convolutions $u*\rho_\eps$ sont définies sur~$P$ pour $\eps$ assez petit,
lisses, et convergent uniformément vers~$u$ lorsque $\eps\to 0$.
Cela conclut la démonstration de~\emph a).

\item
Soit $\mathscr P$ le sous-faisceau de~$\mathscr C_P$
dont les sections sur un ouvert~$U$ sont
les fonctions continues~$u$ sur~$U$ telles que $u\circ f|_{Q\cap f^{-1}(U)}$ 
soit convexe sur~$Q\cap f^{-1}(U)$.
Il contient les fonctions lisses~$u$ telles que $\ddc u\geq 0$,
et est stable par limite uniforme. Il contient donc toutes les fonctions convexes sur~$P$.
\qedhere
\end{enumerate}
\end{proof}

\def\GR{G\textsubscript{$\R$}}
\def\signe{\operatorname{signe}}
\begin{lemm}\label{lemm.extension-G-lisse}
\begin{enumerate}
\item
Soit $C$ une $\R$-cellule d'un espace affine~$V$
et soit $u$ une fonction \GR-lisse sur~$\partial(C)$.
Il existe une fonction \GR-lisse sur~$C$ qui prolonge~$u$.
\item
Soit $K$ un sous-corps de~$\R$, 
soit $C$ une $K$-cellule \emph{compacte} d'un espace affine~$V$
et soit $u$ une fonction $K$-paralinéaire sur~$\partial(C)$.
Il existe une fonction $K$-paralinéaire sur~$V$ qui prolonge~$u$.
\end{enumerate}
\end{lemm}
\begin{proof}
\begin{enumerate}
\item
Choisissons une décomposition $\R$-cellulaire~$\mathscr D$ de~$\partial(C)$
telle que la fonction~$u$ soit lisse en restriction à chaque cellule.
Soit $o$ un point de~$\mathring C$.
Pour chaque cellule~$D$ de dimension~$n-1$ de~$\mathscr D$,
notons~$Q_D$ l'enveloppe convexe de~$D\cup\{o\}$.
Tout point $x\in Q_D$ s'écrit de manière unique 
$x=(1-\lambda) o + \lambda y$, avec $\lambda\in[0;1]$ et $y\in D$;
posons $v_D(x)=\lambda u(y)$. La fonction~$v_D$ est lisse sur~$Q_D$.
Les fonctions~$v_D$ se recollent en une fonction \GR-lisse $v$ 
sur~$C$ qui convient.

\item
Lorsque $K=\Q$, l'énoncé est une conséquence immédiate du théorème~3.13
de~\cite{DucrosHrushovskiLoeserEtAl-2024}
qui affirme que $u$ s'exprime comme une combinaison min-max
de l'ensemble des fonctions $K$-affines qui apparaissent
dans une décomposition cellulaire de~$C$ adaptée à~$u$.
On vérifie que la démonstration qui y est donnée
vaut verbatim lorsqu'on y remplace
\emph{groupes abéliens ordonnés divisibles}
par \emph{$K$-espaces vectoriels ordonnés}; 
la référence aux décompositions cellulaires linéaires de~\cite{Eleftheriou-2018}
ne pose pas de problème puisque ce dernier article en démontre
l'existence dans cette généralité.\qedhere
\end{enumerate}
\end{proof}

\begin{rema}
Lorsque la $\R$-cellule~$C$ est rationnelle, nous ne savons
pas si l'on peut construire une fonction G-lisse sur~$C$
qui la prolonge. Par notre approche, 
nous nous heurtons à l'existence de cellules rationnelles
sans décomposition \emph{simpliciale} rationnelle.
\end{rema}

\begin{prop}\label{prop.extension-G-lisse}
Soit $P$ un sous-espace paralinéaire d'un espace affine~$V$.
\begin{enumerate}
\item
Soit $u$ une fonction \GR-lisse sur~$P$. Il existe
une fonction \GR-lisse sur~$V$ qui prolonge~$u$.
\item
Soit $K$ un sous-corps de~$\R$; supposons que $P$ soit $K$-paralinéaire
et soit $u$ une fonction $K$-paralinéaire sur~$P$.
Il existe une fonction $K$-paralinéaire sur~$V$ qui prolonge~$u$.
\end{enumerate}
\end{prop}
\begin{proof}
Dans le cas~\emph a), posons $K=\R$.
Les démonstrations sont alors analogues compte tenu
du lemme~\ref{lemm.extension-G-lisse}.
Choisissons en effet une décomposition cellulaire~$\mathscr C$ de~$Q$
en $K$-cellules compactes qui est adaptée à~$P$.
Pour tout entier~$d$, soit $Q_d$ la réunion de~$P$
et des cellules de~$\mathscr C$ de dimension~$\leq d$.
On construit par récurrence sur~$d$ un prolongement \GR-lisse~$v$ 
(resp. $K$-paralinéaire) de~$u$ à~$Q_d$. 
Pour $d=0$, on pose $v(x)=0$ pour tout sommet~$x\in\mathscr C$
tel que $x\not\in P$ et $v(x)=u(x)$ pour tout $x\in P$.
Supposons $d\geq 1$.
Soit~$v$ un prolongement \GR-lisse (resp. $K$-paralinéaire) de~$u$ à~$Q_{d-1}$.
Soit~$C\in\mathscr C$ une cellule de dimension~$d$ qui n'est pas contenue
dans~$P$ ; par hypothèse, la cellule~$C$ est compacte
et la fonction~$v$ est définie sur~$\partial(C)$.
D'après le lemme~\ref{lemm.extension-G-lisse},
il existe une fonction~$v_C$ sur~$C$ qui est \GR-lisse (resp. $K$-paralinéaire)
et prolonge~$v|_{\partial(C)}$.
Les fonctions~$v_C$ se recollent en une fonction~$w$ sur~$Q_d$
qui est \GR-lisse (resp. $K$-paralinéaire).
\end{proof}

\begin{prop}\label{prop.G-lisse-d-convexe}
Soit $P$ un sous-espace paralinéaire compact et convexe d'un espace affine~$V$.

\begin{enumerate}
\item 
Soit $u$ une fonction \GR-lisse sur~$P$.
Il existe des fonctions \GR-lisses~$u_1$ et~$u_2$
sur~$P$, convexes,   telles que $u=u_1-u_2$.
\item
Soit $K$ un sous-corps de~$\R$ et soit $u$ une fonction $K$-paralinéaire
sur~$P$. Il existe des fonctions $K$-paralinéaires~$u_1$  et~$u_2$
sur~$P$, convexes, telles que $u=u_1-u_2$.
\end{enumerate}
\end{prop}
\begin{proof}
Les démonstrations sont analogues.
Dans le cas~\emph a), on pose $K=\R$.
Quitte à remplacer~$V$ par le sous-espace affine engendré par~$P$,
on suppose que $P$ est d'intérieur non vide,
de sorte que le mot « convexe » coïncide avec son sens classique.
Posons $n=\dim(P)=\dim(V)$.
Munissons aussi~$V$ d'une norme euclidienne.

Choisissons une décomposition cellulaire~$\mathscr C$  de~$P$ 
telle que $u$ soit lisse (resp. $K$-paralinéaire) 
en restriction à chaque cellule de~$P$.

Dans le cas~\emph a),
pour tout $c\in\R_+$, la fonction $u+ c \norm{x^2}$ sur~$P$ 
est \GR-lisse ; elle  est lisse sur chaque cellule~$C$ de~$\mathscr C$,
et convexe sur~$C$ si $c$ est assez grand.
Cela permet de supposer que la restriction de~$u$ à chaque cellule~$C$
de~$\mathscr C$ est convexe.

Dans le cas~\emph b), la fonction $u|_C$ est affine,
donc convexe, pour toute cellule $C\in\mathscr C$.

Soit alors~$\mathscr E$ l'ensemble des hyperplans affines engendrés
par les $(n-1)$-cellules de~$\mathscr C$; pour tout $E\in\mathscr E$,
choisissons une forme $K$-affine~$f_E$ de lieu des zéros~$E$.
Pour toute famille~$\lambda=(\lambda_E)_{E\in\mathscr E}$ 
de nombres réels positifs,
posons $v_\lambda =u + \sum_E \lambda_E \abs{f_E}$; nous
allons montrer que $v_\lambda$ est convexe
si les~$\lambda_E$ sont assez grands.

Il s'agit de vérifier
que la fonction $t\mapsto v_\lambda((1-t)a + t b)$ 
de~$[0;1]$ dans~$\R$ est convexe
pour tout choix $(a,b)$ de points de~$P$.
Nous allons commencer par le prouver lorsque le segment $[a;b]$
évite toute cellule de codimension~$\geq 2$.

Fixons deux  tels points~$a,b$ et une famille~$\lambda=(\lambda_E)$ 
et notons $f$ cette fonction.
Pour toute cellule~$C$ de~$\mathscr C$,
notons $I_C$ l'intervalle de~$[0;1]$ formé des $t$ tels
que $(1-t)a+tb \in C$. Ces intervalles recouvrent~$[0;1]$
et la fonction~$f$ est lisse et convexe sur chacun d'entre eux.

Soit $t$ un point de~$\mathopen]0;1\mathclose[$;
posons $c=(1-t)a+tb$.
Les dérivées à droite et à gauche de~$f$ sont croissantes
dans des voisinages à droite et à gauche de~$t$.
Il reste à vérifier que la dérivée à gauche de~$f$ en~$t$ est plus petite
que sa dérivée à droite.

Soit $C,C'$ des  cellules de dimension~$n$ de~$\mathscr C$
contenant~$c$ 
telles que $C$ rencontre le segment~$\mathopen]c;b]$
et $C'$ rencontre le segment~$[a;c\mathclose[$.
Si $C=C'$, alors $f$ a même dérivée à droite et à gauche en~$t$.

Sinon, le point~$t$ appartient à une unique $(n-1)$-cellule~$D$
commune à~$C$ et~$C'$, car le segment~$[a;b]$ évite les cellules
de codimension~$\geq 2$; il reste à s'assurer que $f'_g(t)\leq f'_d(t)$.
On a
\[  f'_g(t) =  \langle \nabla u_{C'}(c), b-a\rangle + \sum_E \lambda_E
   \langle \nabla f_E, b-a\rangle \signe(f_E((1-t^-)a+t^-b))  \]
et
\[  f'_d(t) =  \langle \nabla u_C(c), b-a\rangle + \sum_E \lambda_E
   \langle \nabla f_E, b-a\rangle \signe(f_E((1-t^+)a+t^+b))) \]
où l'on désigne par~$t^+$ et~$t^-$ des points suffisamment proches de~$t$
tels que $t^-<t<t^+$, et la fonction signe vaut~$1$ sur~$\R_+^*$,
$0$ en~$0$ et~$-1$ sur~$\R_-^*$.
Alors,
\begin{multline*}
 f'_d(t)-f'_g(t) = \langle \nabla u_C(c)-\nabla u_{C'})(c),b-a\rangle  \\
{} + \sum_E \lambda_E \langle \nabla f_E, b-a\rangle 
\left( \signe(f_E((1-t^+)a+t^+b))-\signe (f_E((1-t^-)a+t^-b))\right. \end{multline*}
Dans cette formule, la différence de  signes
vaut~$0$ si le segment~$s\mapsto (1-s)a+sb$ ne traverse pas
l'hyperplan affine~$\langle E\rangle$ en~$t$;
dans le cas contraire elle vaut~$2$ si on $\langle\nabla f_E,b-a\rangle >0$
et $-2$ dans le cas contraire; notons-la $2\eps_E$. 
Ainsi, seul l'hyperplan $E=\langle D\rangle$ subsiste dans la somme
et l'on a
\begin{align*}
 f'_d(t)-f'_g(t) & = \langle \nabla u_C(c)-\nabla u_{C'}(c),b-a\rangle 
+ 2 \eps_E \lambda_E \langle \nabla f_E, b-a\rangle  \\
& = \langle \nabla u_C(c)-\nabla u_{C'}(c)+2\eps_E\lambda_E \nabla f_E , b-a \rangle .
\end{align*}

Sur tout point~$x$ de~$D$,
la forme linéaire $\nabla u_C(x)-\nabla u_{C'}(x)$ est nulle
sur l'espace directeur~$E$ de~$\langle D\rangle$, 
car $u_C$ et $u_{C'}$ coïncident sur~$D$,
donc cette forme linéaire est proportionnelle à la partie linéaire de~$f_E$,
qui est $\nabla f_E$.
Pour tout point $x\in D$, on peut ainsi écrire
\[ \nabla u_C(x)-\nabla u_{C'}(x) = \mu_{D}(x) \nabla f_E \]
où $\mu_{D}\colon D\to \R$ est continue. 
La condition $f'_g(t)\leq f'_d(t)$ est donc vérifiée
si l'on a $\eps_{E} \mu_{D}(c)+ 2\lambda_E \geq 0$.
Comme la décomposition cellulaire~$\mathscr C$ est finie
et constituée de cellules compactes, ce sera même le cas dès que
chaque terme de~$\lambda$ est assez grand,
indépendamment du choix du couple de points~$(a,b)$
tel que le segment~$[a;b]$
évite les cellules de codimension~$\geq 2$. 

Choisissons un tel~$\lambda$ et vérifions que $v$ est convexe.
Soit $(a,b)$ un couple de points de~$P$.
On peut trouver $a'$ et~$b'$
dans~$P$, arbitrairement proches de~$a$ et~$b$ respectivement
tels que le segment~$[a';b']$
évite toute cellule de codimension au moins~$2$.
Par passage à la limite, la convexité sur ces segments~$[a';b']$
entraîne la convexité sur le segment~$[a;b]$.
\end{proof}

\def\pf{\text{\upshape pf}}
\def\Rad{\text{\upshape Rad}}
\def\Bor{\text{\upshape Bor}}
\def\Lunloc{{\mathrm L^1_{\mathrm{loc}}}}
\def\pr{\operatorname{pr}}
\def\MAT{\mathrm M}
\def\ShHom{\mathop{\mathscr H\!\mathit{om}}\nolimits}

\chapter{Espaces paralinéaires (\textsc{pl})}

\section{Notion d'espace paralinéaire}

%
 
\subsection{}
Soit $X$ un espace topologique séparé.
Une \emph{carte paralinéaire}\index{paralinéaire!carte ---} sur~$X$ 
est une partie localement fermée~$P$ de~$X$
munie d'une application continue~$\phi$ vers un espace affine réel~$V$
qui induit un homéomorphisme de~$P$ sur un sous-espace paralinéaire
de~$V$ ; la partie~$P$ est appelée domaine de la carte.
Une telle carte est dite \emph{cellulaire}
si l'image de~$P$ par~$\phi$ est une cellule (fermée) de~$V$.

Deux cartes paralinéaires $(P,\phi)$ et $(Q,\psi)$ sont
dites compatibles si $(P\cap Q,\phi|_{P\cap Q})$
et $(P\cap Q,\psi|_{P\cap Q})$ sont des cartes paralinéaires.

Un \emph{atlas paralinéaire} sur~$X$ est une famille $((P_i,\phi_i))_i$
de cartes paralinéaires sur~$X$, deux à deux compatibles,
telle que la famille $(P_i)$ soit un  G-recouvrement de~$X$.

On dit que deux atlas sur~$X$ sont équivalents si leur réunion
est un atlas, ce qui revient à dire que chaque carte  
de l'un est équivalent à toute carte de l'autre.

Si $X$ est un espace topologique muni d'un atlas,
la réunion de tous les atlas 
qui lui sont équivalents est un atlas sur~$X$, maximal.

\begin{defi}
Un \emph{espace paralinéaire}\index{paralinéaire!espace ---}
est un espace topologique séparé
muni d'une classe d'équivalence d'atlas paralinéaires.
\end{defi}

Lorsque nous parlerons de \emph{carte paralinéaire}, 
voire de \emph{carte}, en particulier de carte cellulaire,
d'un espace paralinéaire~$X$,
il sera sous-entendu que cette carte
appartient à l'un des atlas équivalents définissant
la structure de~$X$.\index{cellule!d'un espace paralinéaire}

Un espace affine~$V$ possède une structure d'espace paralinéaire 
évidente, donnée par l'atlas $(V,\id_V)$.

\subsection{}
Le produit de deux espaces paralinéaires~$X$ et~$Y$ possède
une structure naturelle d'espace paralinéaire.
Choisissons des atlas sur~$X$ et~$Y$ qui définissent
leur structure d'espace paralinéaire; on définit comme suit
un atlas paralinéaire sur $X\times Y$.
Ses cartes sont les applications de la forme 
$(\phi,\psi)\colon P\times Q\to E \times F$,
où $\phi\colon P\to E$ et $\psi\colon Q\to F$ sont des cartes
paralinéaires sur~$X$ et~$Y$ respectivement.
La structure d'espace paralinéaire définie par cet
atlas ainsi obtenu ne dépend  que des classes d'équivalence
des atlas choisis sur~$X$ et~$Y$.

\subsection{}
Soit $X$ un espace paralinéaire et soit $Y$ une partie de~$X$.
On dit que $Y$ est un \emph{sous-espace paralinéaire} de~$X$
s'il est G-recouvert par des domaines de cartes  de~$X$ qu'il contient.
Lorsque $X$ est un espace affine muni de son atlas canonique, 
on retrouve les notions définies en~\ref{def.sous-espace-pl-V}.

\subsection{}

 Soit $X$ un espace paralinéaire et soit $\mathscr A=((P_i,f_i))$ un atlas paralinéaire sur~$X$.
Soit $Y$ une partie de~$Y$.
 On dit que $Y$ est $\mathscr A$-polyédrale (ou est un $\mathscr A$-polyèdre)
 si, pour tout~$i$, la partie $f_i(P_i\cap Y)$ est un polyèdre
 de l'espace affine but de la carte~$f_i$. 
 L'intersection de deux $\mathscr A$-polyèdres,
 la réunion d'une famille finie de $\mathscr A$-polyèdres
 fermés sont des $\mathscr A$-polyèdres.

 Une partie $\mathscr A$-polyédrale est un sous-espace paralinéaire fermé.

  Soit $Y$ une partie $\mathscr A$-polyédrale de~$X$
et soit $h$ une application de~$Y$ à valeurs dans un espace affine~$W$.
On dit que $h$ est $\mathscr A$-polyédrale si, pour tout~$i$,
l'application de $f_i(P_i\cap Y)$ dans~$W$ déduit de $h|_{P_i}$
est polyédrale.

\subsection{}
Soit $Y$ un sous-espace paralinéaire de~$X$.
Alors, $Y$ est localement fermé dans~$X$.
De plus, la famille des cartes paralinéaires sur~$X$ dont le domaine
est contenu dans~$Y$ est un atlas paralinéaire de~$Y$
par lequel on le munit d'une structure d'espace paralinéaire.

Dans une cellule d'un espace affine, tout point possède une base
de voisinages formée de cellules. On en déduit  qu'un ouvert
d'une cellule est G-recouvert par les cellules qu'il contient,
puis qu'un ouvert d'un espace paralinéaire en est un sous-espace
paralinéaire.

L'intersection de deux sous-espaces paralinéaires 
est un sous-espace paralinéaire ;
la réunion d'une famille localement finie de sous-espaces
paralinéaires fermés est un sous-espace paralinéaire fermé.

Tout point d'un espace paralinéaire possède une base de voisinages
qui sont des sous-espaces paralinéaires compacts.

Un espace paralinéaire est paracompact si et seulement
si chacune de ses composantes connexes est réunion dénombrable
de domaines de cartes cellulaires compacts.

\begin{lemm}\label{lemm.pl-paracompact}
Tout sous-espace paralinéaire d'un espace paralinéaire 
dénombrable à l'infini (resp. paracompact)
est dénombrable à l'infini (resp. paracompact).
\end{lemm}
\begin{proof}
Soit $X$ un espace paralinéaire 
et soit $Y$ un sous-espace paralinéaire de~$X$.
Supposons que $X$ est dénombrable à l'infini et démontrons
que $Y$ l'est.
Par hypothèse, $X$ est réunion d'une famille
dénombrable $(P_n)$ de domaines de cartes compacts. 
Pour tout~$n$, $Y\cap P_n$ est alors homéomorphe à une partie 
localement fermée d'un espace numérique, donc est dénombrable à l'infini.
Par suite, $Y$  est dénombrable à l'infini.

Supposons maintenant que $X$ est paracompact et démontrons que $Y$ l'est.
D'après ce qui précède,
la trace de~$Y$ sur toute composante connexe de~$X$  est 
dénombrable à l'infini.
Il en résulte que $Y$ est paracompact.
\end{proof}
\begin{coro}\label{coro.voisinage}
Soit $X$ un espace paralinéaire et soit $Y$
un sous-espace paralinéaire de~$X$.
Si $Y$ est dénombrable à l'infini,
il possède un voisinage ouvert dans~$X$
qui est dénombrable à l'infini, et en particulier paracompact.
\end{coro}
\begin{proof}
Supposons que $Y$ soit dénombrable à l'infini et
soit $(Q_n)$ une famille dénombrable de parties compactes
recouvrant~$Y$. Pour tout~$n$, $Q_n$ possède un voisinage~$Q'_n$ dans~$X$
qui est un sous-espace paralinéaire compact.
Son intérieur est dénombrable à l'infini 
d'après le lemme~\ref{lemm.pl-paracompact}. 
La réunion~$Y'$ des~$Q'_n$ est un voisinage de~$Y$ dans~$X$
qui est dénombrable à l'infini.
\end{proof}
\subsection{}
Soit $X$ un espace paralinéaire. C'est un espace topologique.
Nous lui associons aussi un site~$X_{\mathrm G}$;
sa catégorie sous-jacente est formée des sous-espaces paralinéaires 
et des inclusions entre iceux;
elle est munie de la G-topologie pour laquelle 
les recouvrements sont les G-recouvrements.

\subsection{}
Soit $X$ et $Y$ des espaces paralinéaires.
On dit qu'une application continue $f\colon X\to Y$ est paralinéaire\index{paralinéaire!application ---} 
si $X$ est G-recouvert par des domaines de cartes $(P,\phi)$ 
tels que $f(P)$ soit contenu dans le domaine~$Q$
d'une carte~$(Q,\psi)$ de~$Y$ 
et tel que l'application  $\psi\circ f\phi^{-1}|_{\phi(P)}$
de~$\phi(P)$ dans~$\psi(Q)$ induite par~$f$ 
soit la restriction d'une application affine.

L'inclusion d'un sous-espace  paralinéaire 
est une application  paralinéaire,
de même que la composée de deux applications paralinéaires.
Les espaces paralinéaires forment donc une catégorie.

La bijection réciproque d'une application paralinéaire bijective est 
paralinéaire (cela est dû au fait que dans ce texte, nous n'imposons
à la partie linéaire
d'une application affine que d'être à coefficients dans un 
\emph{sous-corps} de~$\R$).
Une application paralinéaire est donc un isomorphisme si et seulement
si elle est bijective.

\subsection{}
On dit qu'une application paralinéaire $f\colon X\to Y$
est une \emph{immersion par morceaux}
s'il existe un G-recouvrement $(P_i)$ par des sous-espaces
paralinéaires de~$X$ 
telle que pour tout~$i$, l'application~$f$ induise 
un isomorphisme de~$P_i$ sur un sous-espace paralinéaire de~$Y$.

\begin{prop}\label{prop.sections-G-locales}
Soit $f\colon X\to Y$ un morphisme d'espaces paralinéaires.
Si $f$ est compact, alors $f(X)$ est un sous-espace paralinéaire 
fermé de~$Y$ sur lequel $f$ possède G-localement des sections
paralinéaires.
\end{prop}
\begin{proof}
La question étant G-locale sur~$Y$, on peut supposer
que $Y$ est compact; l'espace~$X$ est alors compact.
Considérons des G-recouvrements finis $(P_i)$ de~$X$
par des domaines compacts de cartes $(P_i,\phi_i)$ 
et $(Q_j)$ de~$Y$ par des domaines compacts de cartes $(Q_j,\psi_j)$,
et, pour tout~$i$, un indice~$j_i$ 
tel que $f(P_i)\subset Q_{j_i}$ et que l'application
induite de $\phi_i(P_i)$ dans $\psi_{j_i}(Q_{j_i})$
soit la restriction d'une application affine.

D'après le corollaire~\ref{coro.image-polyedre-sections},
$f(P_i)$ est un sous-espace paralinéaire de~$Q_{j_i}$
sur lequel~$f$ possède G-localement des sections.
La proposition s'ensuit.
\end{proof}

\subsection{}
Soit $X$ un espace paralinéaire.

On appelle dimension de~$X$ la borne supérieure $\dim(X)$ des dimensions
des domaines de cartes  de~$X$.
On a $\dim(\emptyset)=-\infty$.

Soit $x$ un point de~$X$; on appelle dimension de~$X$ en~$x$
la borne supérieure $\dim_x(X)$  des dimensions en~$x$
des domaines de cartes de~$X$ contenant~$x$.
On a $\dim(X)=\sup_{x\in X} \dim_x(X)$.

Soit $n\in\N$. On note $X^{(n)}$ l'ensemble des points~$x\in X$
tels que $\dim_x(X)\geq n$. C'est un sous-espace paralinéaire fermé de~$X$.
Il est G-recouvert par les domaines
de cartes cellulaires de~$X$ de dimension~$\geq n$.

\begin{defi}\label{PL-negligeable}
Soit $d$ un entier.
Un sous-ensemble d'un espace paralinéaire~$X$
sera dit
\emph{$d$-négligeable}\index{partie $d$-négligeable}
si sa trace sur tout domaine de carte de~$X$
est contenue dans un sous-espace paralinéaire de dimension~$<d$.
\end{defi}
Il suffit de vérifier cette propriété pour toute cellule 
d'un G-recouvrement donné de~$X^{(d)}$.
Si $X$ est paracompact et  si $A$ est une partie $d$-négligeable de~$X$,
il existe un sous-espace paralinéaire~$Y\subset X$, de dimension~$<d$,
tel que $A\subset Y$.

\section{Décompositions cellulaires}

%
%
 
\subsection{}
Soit $X$ un espace paralinéaire.  

Soit $C$ une partie de~$X$.
On dit que des cartes cellulaires $(C,f)$ et $(C,g)$ de domaine~$C$
sont équivalentes si l'application $g\circ f^{-1}$ de~$f(C)$ dans~$g(C)$
est la restriction d'une application affine.
On appellera \emph{cellule}\index{cellule!d'un espace paralinéaire}
de~$X$ une classe d'équivalence de telles
cartes cellulaires.
Une carte cellulaire d'une cellule sera appelée \emph{carte affine}.

En pratique, on se permettra d'écrire « Soit $(C,f)$ une cellule... »,
en choisissant une carte affine privilégiée,
voire, en sous-entendant le choix d'une telle carte, « soit $C$ une cellule ».
On dira aussi que $C$ est une $f$-cellule.

Sur une cellule de~$X$, on dispose
alors de la notion d'application affine à valeurs
dans un espace affine, voire dans une cellule. On dispose
aussi de la notion de barycentre, de convexité,...

Cette terminologie est incompatible avec celle introduite
dans le cas d'un espace affine~$V$ où les cellules introduites
initialement sont celles de la forme $(C,j)$, où $j $ est l'inclusion.
Pour éviter toute confusion, ces cellules seront appelées
\emph{cellules fortes} de~$V$.\index{cellule forte}

\subsection{}
Soit $C$ une cellule de~$X$ et soit $f$ une carte affine de~$C$. 
On appelle \emph{intérieur} de~$C$ l'image réciproque par~$f$ de l'intérieur
de~$f(C)$; c'est une partie connexe, ouverte et dense de~$C$.
C'est aussi
l'ensemble~$\mathring C$ des points en lesquels $C$ est une variété
topologique. 

On appelle \emph{bord} de~$C$ le complémentaire $\partial(C)$
de son intérieur; c'est un sous-espace paralinéaire de~$C$;
il est vide si $C$ est réduit à un point, 
et purement de dimension~$\dim(C)-1$ sinon.

Soit $D$ une cellule de~$X$. On dit que c'est une 
\emph{sous-cellule} de~$C$ si 
le domaine de~$D$ est contenu  dans celui de~$C$
et si l'inclusion de~$D$ dans~$C$ est une application affine.

\begin{lemm}
Soit $f\colon C\to D$ un morphisme affine de cellules.
Alors $f(C)=D$ si et seulement si $f(\mathring C)=\mathring D$.
\end{lemm}
\begin{proof}
On peut supposer que $C$ est une cellule forte d'un espace affine~$V$,
que $D$ est une cellule forte d'un espace affine~$W$,
et que $f$ est induite par une application affine de~$V$ sur~$W$.
Comme $C$ et $D$ sont fermées dans~$V$ et~$W$ respectivement,
l'assertion découle du corollaire~\ref{coro.interieur-adherence}.
%
%
\end{proof}

\begin{defi}
On appelle \emph{décomposition cellulaire}\index{décomposition cellulaire} 
de~$X$
une famille localement finie~$\mathscr C$
de cellules fermées de~$X$
telle que :
\begin{enumerate}
\item Si $C$ et~$C'$ sont des cellules de~$\mathscr C$
telles que $C'$ soit contenue (ensemblistement) dans~$C$, 
alors soit $C'=C$,
soit $C'$ est une sous-cellule de~$C$ contenue dans~$\partial(C)$;
\item Le bord de toute cellule de~$\mathscr C$ est réunion
finie de cellules de~$\mathscr C$;
\item L'intersection de deux cellules quelconques de~$\mathscr C$
est une réunion de cellules de~$\mathscr C$.
\end{enumerate}
\end{defi}

Si $X$ est un sous-espace paralinéaire d'un espace affine,
une décomposition cellulaire de~$X$ est dite \emph{forte} 
si elle est constituée de cellules fortes.\index{décomposition cellulaire!forte}

On dit qu'une décomposition cellulaire~$\mathscr C$ de~$X$
est adaptée à un sous-espace paralinéaire fermé~$Y$ de~$X$ 
si $Y$ est réunion de cellules appartenant à~$\mathscr C$.
Plus généralement, si $Y$ est un sous-espace paralinéaire fermé de~$X$
et~$h$ une application paralinéaire de~$Y$ vers un espace affine, voire une cellule, on dit que $\mathscr C$ est adaptée au couple~$(Y,h)$
si $Y$ est réunion de cellules appartenant à~$\mathscr C$
et si, pour toute cellule~$C$ de~$\mathscr C$, 
$h|_Y$ est une application affine.

\begin{rema}
Soit $C$ et $D$ deux cellules
d'une décomposition cellulaire~$\mathscr C$.

\begin{enumerate}
\item
Si $C\cap \mathring D\neq\emptyset$, alors $C\supset D$.
\item
Si $C\neq D$, alors $\mathring C\cap \mathring D=\emptyset$.
\end{enumerate}

Démontrons d'abord~(1). Par définition d'une décomposition
cellulaire, $C\cap D$ est réunion de cellules
de~$\mathscr C$. Elles sont contenues dans~$D$ et l'une d'elles
au moins n'est pas contenue dans le bord de~$D$; elle est donc égale à~$D$.

L'assertion~(2) en découle: si $\mathring C\cap \mathring D$ n'est pas
vide, on a $C\subset D$ et $D\subset C$, donc $C=D$.
\end{rema}

\begin{rema}\label{rema.dec-cellulaire-partition}
Soit $\mathscr C$ une décomposition cellulaire d'un espace
paralinéaire~$X$. 
Soit $x$ un point de~$X$ et soit $C$ une cellule de~$\mathscr C$
de dimension minimale contenant~$x$. Alors, $x\in\mathring C$.
(Dans le cas contraire, $x$ appartiendrait au bord de~$C$
qui est réunion de cellules de dimension strictement plus petites.)
Cela prouve que la famille $(\mathring C)_{C\in\mathscr C}$
recouvre~$X$. D'après la remarque précédente, c'est une partition de~$X$.
\end{rema}

\begin{rema}\label{rema.dec-cellulaire-finitude}
Ces axiomes entraînent que chaque cellule de~$\mathscr C$
ne rencontre qu'un nombre fini d'autres cellules. 

Considérons en effet une  cellule~$C$; posons $d=\dim(C)$
et supposons par
récurrence que chaque cellule de dimension~$<d$ ne rencontre
qu'un nombre fini de cellules de~$\mathscr C$.
Si $d=0$, $C$ est réduit à un point, qui  n'appartient qu'à un nombre
fini de cellules puisque $\mathscr C$ est localement finie.
Supposons $d>0$. 
Soit $C'$ une cellule de~$\mathscr C$ qui rencontre~$C$;
soit $D$ une cellule de~$\mathscr C$ qui est contenue dans~$C\cap C'$.
Soit $C=D$, soit $D\subset\partial(C)$; dans les deux cas, 
$D$ rencontre~$\partial(C)$, donc $C'$ aussi.

Par définition d'une décomposition cellulaire,
$\partial (C)$ est réunion d'une famille
finie de cellules de dimension~$<d$; par récurrence chacune
d'elles ne rencontre qu'un nombre fini de cellules de~$\mathscr C$,
d'où l'assertion.
\end{rema}

\begin{lemm}\label{lemm.vois-ouvert-cellule}
Soit $X$ un espace paralinéaire, soit $\mathscr C$
une décomposition cellulaire de~$X$ et soit~$C$ une cellule appartenant à~$\mathscr C$.
Notons $W_C$ la réunion des cellules ouvertes~$\mathring D$,
où $D$ parcourt l'ensemble des cellules de~$\mathscr C$ contenant~$C$.
Alors $W_C$ est un voisinage ouvert de~$\mathring C$ dans~$X$
et son complémentaire est un sous-espace paralinéaire fermé.
\end{lemm}
\begin{proof}
Notons $W'_C$ le complémentaire de~$W_C$ dans~$X$.
Pour qu'une cellule~$D$ ne contienne pas~$C$,
il faut et il suffit qu'elle soit disjointe de~$\mathring C$,
si bien que $W'_C$ est la réunion des cellules~$D$ de~$\mathscr C$
telles que $D\cap \mathring C=\emptyset$.
Cela prouve que $W'_C$ est un sous-espace paralinéaire fermé de~$X$.
Il est disjoint de~$\mathring C$ par construction.
\end{proof}

\subsection{}\label{sss.dec-raffinement}
Soit $\mathscr C$ et $\mathscr D$  des décompositions cellulaires de~$X$.
On dit que $\mathscr D$ raffine~$\mathscr C$ 
si toute cellule de~$\mathscr D$ 
est une sous-cellule d'une cellule de~$\mathscr C$, 
et si toute cellule de~$\mathscr C$ 
est réunion de cellules de~$\mathscr D$.
Dans ce cas, 
toute cellule de~$\mathscr D$ est sous-cellule de toute cellule
de~$\mathscr C$ qui la contient.

Soit $D$ une cellule de~$\mathscr D$.

Soit $C$ une cellule de~$\mathscr C$;
on a $D\subset C$ ou $C\cap D\subset\partial(D)$. En effet,
$C$ est une union de cellules de~$\mathscr D$, donc $C\cap D$ 
est une union de cellules de~$\mathscr D$, contenue dans~$D$.

On en déduit que si $D$ est contenue dans une union de cellules de~$\mathscr C$,
elle est contenu dans l'une d'elles.

Soit $C\in\mathscr C$ une cellule de dimension minimale contenant~$D$.
Démontrons que $\partial(C)\cap D\subset \partial(D)$.
Par hypothèse, $\partial(C)$ est réunion de cellules~$C_i$ de~$\mathscr C$,
elles sont de dimension~$<\dim(C)$, donc $D\not\subset C_i$.
D'après ce qui précède, $C_i\cap D\subset \partial(D)$,
puis $\partial(C)\cap D\subset\partial(D)$.
En passant aux complémentaires, on a aussi $\mathring D\subset\mathring C$;
notons même que $C$ est l'unique cellule de~$\mathscr C$
telle que $\mathring D\subset\mathring C$.

\begin{lemm}\label{lemm.dec-cellulaire-plonge}
Soit $V$ un espace affine et 
soit~$\mathscr B$ une famille finie de couples $(B,h)$, 
où $B$ est un polyèdre de~$V$ et $h$ une application polyédrale sur~$B$.
Il existe une décomposition cellulaire finie forte de~$V$ qui
est adaptée à tout élément de~$\mathscr B$.

En particulier, il existe, pour tout polyèdre~$P$ de~$V$,
il existe une décomposition cellulaire finie forte de~$P$.
\end{lemm}
\begin{proof}
Quitte à raffiner la famille~$\mathscr B$, on suppose que pour tout
couple~$(B,h)\in\mathscr B$, $B$ est une cellule forte de~$V$ sur
laquelle $h$ est affine. 
Il existe alors une famille finie~$\Phi$ de formes affines
sur~$\R^n$ telle pour tout~$(B,h)$ appartenant à~$\mathscr B$,
la cellule~$B$ soit combinaison booléenne 
de demi-espaces définis par ces formes. 
Pour chaque famille $c=(\Box_\phi)_{\phi\in\Phi}$,
où $\Box_\phi\in\{\mathord\leq,\mathord=,\mathord\geq\}$, 
considérons le polyèdre $B_c$ formé des points $x\in V$
tels que $\phi(x)\mathrel{\Box_\phi} 0$ pour tout $\phi\in\Phi$;
s'il n'est pas vide, c'est une cellule forte de~$V$.
La famille des~$B_c$, tels que $B_c\neq\emptyset $, convient. 
Dans ce cas, l'intersection de deux cellules est ou vide ou une cellule
forte.

La dernière assertion s'en déduit en prenant pour $\mathscr B$ la famille
réduite au couple~$(P,0)$.
\end{proof}

%

\begin{defi}
Soit $X$ un espace paralinéaire et soit $\mathscr A=(X_i,f_i)$
un atlas sur~$X$ et 
soit $\mathscr C$ une décomposition cellulaire de~$X$.
On dit que $\mathscr C$ est adaptée à l'atlas~$\mathscr A$
si les conditions suivantes sont vérifiées:
\begin{enumerate}
\item Tout~$X_i$ est réunion de cellules de~$\mathscr C$;
\item Pour tout~$i$, toute cellule~$C\in\mathscr C$ qui est contenue dans~$X_i$
est une $f_i$-cellule.
\end{enumerate}
\end{defi}

\begin{rema}\label{rema.decomp-cell-f}
Soit $X$ un espace paralinéaire paracompact 
et soit $f\colon X\to\R^m$ une immersion par morceaux.
Il existe un G-recouvrement localement fini~$(X_i)$ de~$X$
tel que $f|_{X_i}$ soit injective, pour tout~$i$.
On peut même choisir les parties~$X_i$ compactes.
Alors, la famille $(X_i,f|_{X_i})$ est un atlas paralinéaire sur~$X$.
Si une décomposition cellulaire est  adaptée à cet atlas, ses cellules
sont des $f$-cellules.
\end{rema}

\begin{prop}\label{prop.decomp-cell}
Soit $X$ un espace paralinéaire 
et soit $\mathscr A=(X_i,f_i)_{1\leq i\leq m}$
un atlas de type fini sur~$X$; soit $\mathscr B$ une famille
finie de couples $(P,h)$, où $P$ est une partie $\mathscr A$-polyédrale
de~$X$ et $h$ une application $\mathscr A$-polyédrale sur~$P$.
Il existe une décomposition cellulaire finie~$\mathscr C$ de~$X$,
qui est adaptée à l'atlas~$\mathscr A$ et adaptée
à tout élément de~$\mathscr B$.
\end{prop}
\begin{proof}
Si $m=0$, l'espace~$X$ est vide et la décomposition cellulaire vide convient.
Si $m=1$, $X=X_1$ est identifié au polyèdre~$f_1(X_1)$ d'un espace affine
et la décomposition cellulaire fournie par le lemme~\ref{lemm.dec-cellulaire-plonge} convient.
On raisonne par récurrence sur le cardinal~$m$ de l'atlas~$\mathscr A$
en supposant le résultat prouvé pour un espace paralinéaire
muni d'un atlas de type fini de cardinal~$<m$.

Supposons $m\geq 2$ et posons $Y=X_1\cup\dots \cup X_{m-1}$ et $Z=X_m$.
Par récurrence, il existe une décomposition cellulaire~$\mathscr C_Y$
de~$Y$  satisfaisant les conditions requises par la proposition
pour l'atlas $(X_i,f_i)_{i<m}$ et la famille des~$(B\cap Y,h|_{B\cap Y}$,
pour $(B,h)$ parcourant~$\mathscr B$,
à laquelle on adjoint le couple~$(Y\cap Z,f_m|_{Y\cap Z})$.
En particulier, toute cellule~$C\in\mathscr C_Y$ contenue
dans $Y\cap Z$ est une $f_m$-cellule.

D'après le cas $m=1$,
il existe aussi une décomposition cellulaire~$\mathscr C_Z$ de~$Z$
formée de $f_m$-cellules
qui est adaptée à la famille $(B\cap Z,h|_{B\cap Z})$
(pour $(B,h)$ dans~$\mathscr B$), à laquelle
on adjoint les cellules de~$\mathscr C_Y$ qui sont contenues dans~$Z$.
Observons que pour toute cellule~$C\in\mathscr C_Z$
qui est contenue dans~$Y$ et tout $i<m$ tel que $C\subset X_i$,
l'application~$f_i\circ f_m^{-1}$  de $f_m(C)$ dans~$f_i(C)$ est affine,
car $C$ est contenue dans une cellule~$C'\in\mathscr C_Y$ 
contenue dans~$Y\cap Z$ et l'application réciproque $f_m\circ f_i^{-1}$
de $f_i(C')$ dans~$f_m(C')$ est affine.

Soit $\mathscr C$ la réunion de~$\mathscr C_Z$ et de l'ensemble des cellules
$C\in\mathscr C_Y$ qui ne sont pas contenues dans~$Z$.
Par construction, toute cellule de~$\mathscr C$
est contenue dans l'un des~$X_i$ et tout $X_i$
est réunion de cellules, de même que tout élément de~$\mathscr B$.
De plus, 
pour tout couple d'indices~$(i,j)$ et toute cellule
$C\subset X_i\cap X_j$, 
l'application~$f_i$ établit un isomorphisme paralinéaire
de~$C$ sur une cellule forte, et l'application $f_j\circ f_i^{-1}$
de~$f_i(C)$ dans~$f_j(C)$ est affine.
Il reste à vérifier que $\mathscr C$ est une décomposition cellulaire
de~$X$.

Les deux propriétés suivantes découlent de la construction:
1) les cellules de~$\mathscr C$ contenues dans~$Z$ forment la décomposition cellulaire~$\mathscr C_Z$ de~$Z$;
2) toute cellule de~$\mathscr C_Y$ qui est contenue dans~$Z$ est réunion de
cellules de~$\mathscr C_Z$;
3) toute cellule de~$\mathscr C_Y$ est réunion de cellules de~$\mathscr C$.

Soit $C$ une cellule de~$\mathscr C$; démontrons que $\partial(C)$
est réunion de cellules de~$\mathscr C$.
C'est évident si $C\subset Z$, d'après ce qui précède.
Sinon, $C\not\subset Z$, donc $C$ est une cellule de~$\mathscr C_Y$;
son bord est donc réunion de cellules de~$\mathscr C_Y$,
donc est réunion de cellules de~$\mathscr C$.

Soit $C$ et $D$ des cellules de~$\mathscr C$. Démontrons que $C\cap D$
est intersection de cellules de~$\mathscr C$.
C'est encore évident si $C$ et $D$ sont toutes deux contenues dans~$Z$.
Si aucune des deux n'est contenue dans~$Z$, alors $C$ et $D$ appartiennent
à~$\mathscr C_Y$, donc $C\cap D$ est réunion de cellules de~$\mathscr C_Y$,
donc est réunion de cellules de~$\mathscr C$.
Supposons enfin que $C\not\subset Z$ et $D\subset Z$.
Comme $Z\cap Y$ est réunion de cellules de~$\mathscr C_Y$,
$C\cap Z$ est réunion de cellules de~$\mathscr C_Y$ qui sont contenues
dans~$Z$, donc de cellules de~$\mathscr C_Z$. Leur intersection avec~$D$
est réunion de cellules de~$\mathscr C_Z$, comme il fallait démontrer.

Soit $C$ et $D$ des cellules de $\mathscr C$ telles que $C\subset D$
et $C\neq D$.
Si $D\subset Z$, alors $C\in\mathscr C_Z$ donc  $C,D\in\mathscr C_Z$,
d'où $C\subset\partial(D)$,
car $\mathscr C_Z$ est une décomposition cellulaire de~$Z$.
Supposons alors que $D\not\subset Z$; on a donc $D\in\mathscr C_Y$.
De plus, $D\cap Z$
est réunion de cellules de~$\mathscr C_Y$ dont aucune n'est égale à~$D$,
elles sont donc contenues dans~$\partial(D)$, si bien que $D\cap Z\subset\partial(D)$.
Si $C\subset Z$, il en résulte que $C\cap D\subset\partial(D)$.
Sinon, on a $C\in\mathscr C_Y$, donc $C\subset\partial(D)$, car $\mathscr C_Y$
est une décomposition cellulaire de~$Y$.
\end{proof}

\begin{coro}\label{coro-decomp-cell-compact}
Soit $X$ un espace paralinéaire compact 
et soit $\mathscr B$ une famille finie  couples $(P,h)$,
où $P$ est un sous-espace paralinéaire compact
de~$X$ et $h$ une application paralinéaire sur~$P$ à valeurs
dans une cellule.
Il existe une  décomposition cellulaire finie sur~$X$ 
qui est adaptée à~$\mathscr B$.
\end{coro}
\begin{proof}
 Comme $X$ est compact, il possède un atlas de type fini~$\mathscr A$;
les $\mathscr A$-polyèdres sont exactement les sous-espaces
paralinéaires compacts de~$X$ et les applications $\mathscr A$-polyédrales
sur un tel sous-espace sont exactement les applications paralinéaires.
L'assertion résulte alors de la proposition~\ref{prop.decomp-cell}.
\end{proof}

\begin{prop}\label{prop.decomp-cell-paracompact}
Soit $X$ un espace paralinéaire paracompact
et soit $\mathscr B$ une famille localement finie
de couples~$(Y,h)$, où $Y$ est un sous-espace paralinéaire fermé de~$X$
et $h$ une application paralinéaire de~$Y$ dans une cellule.
Il existe une décomposition cellulaire $\mathscr C$ de~$X$,
formée de cellules compactes,
qui est adaptée 
à la famille~$\mathscr B$.
\end{prop}
\begin{proof}
On peut supposer que $X$ est connexe et donc réunion
d'une famille dénombrable $(P_n)$ de parties paralinéaires compactes. 
Puisque tout point de~$X$ possède un voisinage qui est un sous-espace 
paralinéaire compact, on en déduit qu'il existe une telle
suite où, pour tout~$n$,
 $P_n$ contenu dans l'intérieur de~$P_{n+1}$.

Posons $Q_0=P_0$. Pour tout~$n\geq 1$, 
notons $Q_n$  l'adhérence de~$P_n\setminus P_{n-1}$; 
c'est un sous-espace paralinéaire de~$X$, 
comme on le vérifie G-localement, ce qui permet de se ramener au cas plongé;
il est de plus compact.
Si $n$ et~$m$ sont des entiers tels que $\abs{n-m}\geq 2$, 
on a $Q_n\cap Q_m=\varnothing$.

Il n'est pas restrictif de supposer que pour tout~$n$, les
couples $(Q_n,0)$ et $(Q_n\cap Q_{n+1},0)$ appartiennent à~$\mathscr B$.

Pour tout~$n$, choisissons une décomposition cellulaire~$\mathscr C_n$ de~$Q_n$
qui est adaptée à la trace sur~$Q_n$ de la famille~$\mathscr B$.
Pour tout~$n$, choisissons une décomposition cellulaire~$\mathscr D_{n,n+1}$
de l'intersection $Q_n\cap Q_{n+1}$ qui est 
adaptée 
à la trace des décompositions~$\mathscr C_n$ et $\mathscr C_{n+1}$,
ainsi qu'à la trace de la famille~$\mathscr B$.

Soit $\mathscr C$ l'ensemble des cellules~$C$ satisfaisant l'une des
deux conditions suivantes, exclusives l'une de l'autre : 
\begin{enumerate}
\item Il existe $n$ tel que $C$ appartienne à $\mathscr C_n$,
et ne soit contenue ni dans~$Q_{n-1}$, ni dans~$Q_{n+1}$ ;

\item Il existe $n$ tel que 
$C\in\mathscr D_{n,n+1}$.
\end{enumerate} 
Par construction, toute cellule appartenant à~$\mathscr C_n$ est réunion
de cellules appartenant à~$\mathscr C$.
On démontre alors comme dans la preuve de la proposition précédente
que $\mathscr C$ est une décomposition cellulaire de~$X$.
Elle vérifie les conditions requises.
\end{proof}

\begin{coro}
Soit $X$ un espace linéaire paracompact.
Toute décomposition cellulaire de~$X$ possède un raffinement
en cellules compactes.
\end{coro}
\begin{proof}
Il suffit d'appliquer la proposition précédente à l'espace~$X$ et 
sa décomposition cellulaire donnée.
\end{proof}

%

\begin{rema}
Soit $X$ et~$Y$ des espaces paralinéaires.

Soit $\mathscr C$ une décomposition cellulaire de~$X$,
soit $\mathscr D$ une décomposition cellulaire de~$Y$.
Pour toute cellule $C\in\mathscr C$ et toute cellule $D\in\mathscr D$,
la partie $C\times D$ est une cellule de~$X\times Y$;
son intérieur est $\mathring C\times\mathring D$
et son bord est $(\partial(C)\times D)\cup (C\times \partial(D))$.
La famille des $C\times D$ est alors une décomposition cellulaire
de~$X\times Y$.
\end{rema}

\section{Décompositions cellulaires relatives}

\begin{defi}
Soit $X$ et $Y$ deux espaces paralinéaires
et soit $f\colon  X \to Y$ une application paralinéaire.
Soit $\mathscr C$ une décomposition cellulaire de~$X$ et
soit $\mathscr D$ une décomposition cellulaire de~$Y$.
On dit que le couple $(\mathscr C,\mathscr D)$
est adapté à~$f$ si les conditions suivantes sont satisfaites:
\begin{enumerate}
\item Pour toute cellule $C\in\mathscr C$, $f$ induit une
application affine surjective de~$C$ sur une cellule de~$\mathscr D$.
\item Pour toute cellule~$D$ de~$\mathscr D$, $f^{-1}(D)$
est réunion de cellules de~$\mathscr C$.
\end{enumerate}
\end{defi}

\begin{rema}\label{rema.relative-restriction}
Soit $f\colon X\to Y$ une application paralinéaire entre espaces
paralinéaires, soit $\mathscr C$ une décomposition cellulaire de~$X$ et
soit $\mathscr D$ une décomposition cellulaire de~$Y$.
On suppose que le couple $(\mathscr C,\mathscr D)$ est adapté à~$f$.
Soit $Y'$ un sous-espace paralinéaire de~$Y$ qui est réunion 
d'une sous-famille $\mathscr D'$ de cellules de~$\mathscr D$; 
soit $X'=f^{-1}(Y)$ et soit $\mathscr C'$ l'ensemble
des cellules de~$\mathscr C$ dont l'image est contenue dans~$Y'$.
Alors 
$\mathscr C'$ est une décomposition cellulaire de l'espace paralinéaire~$X'$,
$\mathscr D'$ est une décomposition cellulaire de l'espace paralinéaire~$Y'$
et le couple $(\mathscr C',\mathscr D')$ est adapté à l'application
paralinéaire déduite de~$f$ par passage aux sous-espaces.
\end{rema}

\begin{lemm}
Soit $X$ et $Y$ des espaces paralinéaires paracompacts
et soit $f\colon  X \to Y$ une immersion par morceaux.
Soit $(\mathscr C,\mathscr D)$ un couple de décomposition
cellulaires qui est adapté  à~$f$.
Alors, pour toute cellule~$D$ de~$\mathscr D$,
l'application~$f$ fait de $f^{-1}(\mathring D)$
un revêtement de~$\mathring D$ dont les feuillets
sont les cellules ouvertes $\mathring C$,
où $C$ parcourt l'ensemble des cellules de~$\mathscr C$
contenues dans~$f^{-1}(D)$ et de dimension~$\dim(D)$.
\end{lemm}
\begin{proof}
Soit $D$ une cellule de~$\mathscr D$ et soit $C$
une cellule de~$\mathscr C$  contenue dans~$f^{-1}(D)$.
Par hypothèse, $f(C)$ est une cellule de~$\mathscr D$.
Si $\dim(C)<\dim(D)$, on a donc $f(C)\subset\partial(D)$.
Sinon, $f$ induit un isomorphisme de~$C$ sur~$D$ 
et donc un isomorphisme de~$\mathring C$ sur~$\mathring D$;
en outre $\mathring C$ est ouverte dans~$f^{-1}(\mathring D)$.
Le lemme en résulte.
\end{proof}

\begin{prop}\label{prop.relative-raffinement}
Soit $f\colon X\to Y$ une application paralinéaire 
entre espaces paralinéaires,
soit $\mathscr C$ une décomposition cellulaire de~$X$ et
soit $\mathscr D$ une décomposition cellulaire de~$Y$.
On fait les hypothèses suivantes:
\begin{enumerate}\def\theenumi{\roman{enumi}}
\item Pour toute cellule~$C$ de~$\mathscr C$, il existe
une cellule~$D$ de~$\mathscr D$ telle que $f(C)\subset D$;
\item Pour toute cellule~$C$ de~$\mathscr C$ et toute cellule~$D$
de~$\mathscr D$ telle que $f(C)\subset D$, l'application $f|_C\colon C\to D$
est affine.
\end{enumerate}
Soit $\mathscr D'$ une décomposition cellulaire de~$Y$ qui raffine~$\mathscr D$
et telle que pour toute cellule~$C\in\mathscr C$, $f(C)$ est réunion de cellules.
Soit $\mathscr C'$ la famille des $f^{-1}(D)\cap C$,
où $C$ est une cellule de~$\mathscr D$ et
$D$ est une cellule de~$\mathscr D'$ 
telles que $\mathring D\subset f(\mathring C)$.
Pour tout tel couple $(C,D)$,
$f^{-1}(D)\cap C$ est une sous-cellule de~$C$ et
son intérieur est $f^{-1}(\mathring D)\cap \mathring C$.
De plus, $\mathscr C'$ est une décomposition cellulaire de~$X$
qui raffine~$\mathscr C$
et le couple $(\mathscr C',\mathscr D')$ est adapté à~$f$.
\end{prop}
\begin{proof}
Commençons par observer,
sous les hypothèses de la proposition, 
que pour toute cellule~$C\in\mathscr C$
possède une structure naturelle de cellule, induite
par n'importe quelle cellule $D\in\mathscr D$ contenant~$f(C)$.

Soit $\mathscr E$ l'ensemble des couples $(C,D)$, où $C$ est une cellule de~$\mathscr E$ et $D$ est une cellule  de~$\mathscr D'$ telles que $\mathring D\subset f(\mathring C)$. Notons que pour tout tel couple, $D\subset f(C)$.

Soit $(C,D)$ un élément de~$\mathscr E$.
Soit $\Delta$ une cellule de~$\mathscr D$ telle que $f(C)\subset\Delta$;
par hypothèse, $f|_C\colon C\to\Delta$ est une application affine.
Alors, $f(\mathring C)=f(C)^\circ$ (prop.~\ref{prop.image-convexe}).
Comme $f^{-1}(\mathring D)\cap \mathring C\neq\emptyset$,
on a  $f^{-1}(D)\cap C\neq\emptyset$.
Puisque $\mathring D\subset f(C)\subset \Delta$, on a $D\subset\Delta$,
et $f^{-1}(D)\cap C$ est une sous-cellule de~$C$.
D'après la prop.~\ref{prop.image-convexe},
son intérieur est $f^{-1}(\mathring D)\cap \mathring C$,
car cette intersection n'est pas vide.

Soit $x$ un point de~$X$.
Soit $\Gamma$  l'unique cellule de~$\mathscr C$ telle que $x\in\mathring \Gamma$,
soit $\Delta$ l'unique cellule de~$\mathscr D'$ telle que $f(x)\in\mathring\Delta$.
Comme $f(\Gamma)$ est réunion de cellules de~$\mathscr D'$,
il en résulte que $\Delta\subset f(\Gamma)$.
Puisque $\mathring\Delta \cap f(\mathring\Gamma)$ n'est pas vide,
c'est l'intérieur de $\Delta\cap f(\Gamma)=\Delta$
(prop.~\ref{prop.intersection-convexe}); 
autrement dit, $\mathring \Delta \subset f(\mathring\Gamma)$
et le couple $(\Gamma,\Delta)$ appartient à~$\mathscr E$
et le point~$x$ est situé sur la cellule ouverte correspondante.

Soit $(C_1,D_1)$, $(C_2,D_2)$ des éléments de~$\mathscr E$
tels que $f^{-1}(D_1)\cap C_1 $ est contenu dans $f^{-1}(D_2)\cap C_2$
et rencontre $f^{-1}(\mathring D_2)\cap\mathring C_2$.
Par hypothèse, $f^{-1}(D_1\cap\mathring D_2)\cap (C_1\cap\mathring C_2)$
n'est pas vide ; soit $x$ un point de cette intersection.
Alors $\mathring C_2$ rencontre~$C_1$, donc $C_2\subset C_1$;
de même, $D_2\subset D_1$. Alors, $f^{-1}(D_1\cap\mathring D_2)\cap (C_1\cap \mathring C_2)=f^{-1}(\mathring D_2)\cap \mathring C_2$,
de sorte que $f^{-1}(D_1)\cap C_1$ contient l'intérieur de~$f^{-1}(D_2)\cap C_2$, donc cette cellule toute entière; elles sont ainsi égales.

Soit $(C_1,D_1)$, $(C_2,D_2)$ des éléments de~$\mathscr E$ et
soit $x\in (f^{-1}(D_1)\cap C_1) \cap (f^{-1}(D_2)\cap C_2)$.
Soit $(\Gamma,\Delta)$ l'élément de~$\mathscr E$ associé à~$x$ comme ci-dessus.
Alors $\mathring \Gamma$ rencontre $C_1\cap C_2$, 
donc $\Gamma\subset C_1\cap C_2$; de même, $\Delta\subset D_1\cap D_2$.
Ainsi, $f^{-1}(\Delta)\cap\Gamma$ est une cellule de~$\mathscr C'$
qui contient~$x$ et qui est
contenue dans l'intersection 
$(f^{-1}(D_1)\cap C_1) \cap (f^{-1}(D_2)\cap C_2)$.
Cela prouve que l'intersection de deux cellules de~$\mathscr C'$
est réunion de cellules de~$\mathscr C'$.

Soit $(C,D)$ un élément de~$\mathscr E$ et soit $x$ un point
du bord de la cellule $f^{-1}(D)\cap C$.
Soit $(\Gamma,\Delta)$ l'élément de~$\mathscr E$ associé à~$x$ comme ci-dessus.
Comme $\mathring\Gamma$ rencontre~$C$, on a $\Gamma\subset C$ ;
de même, $\Delta\subset D$.
Comme $x$ n'appartient pas à l'intérieur de~$f^{-1}(D)\cap C$,
on a $D\neq\Delta$ ou $C\neq\Gamma$.
Si $\Gamma\neq C$, alors $\Gamma$ est contenu dans le bord de~$C$,
et $f^{-1}(\Delta)\cap \Gamma\subset f^{-1}(D)\cap\partial(C)
\subset \partial(f^{-1}(D)\cap C)$.
Sinon, on a $\Delta\neq D$, donc $\Delta$ est contenu dans le bord de~$D$
et $f^{-1}(\Delta)\cap\Gamma \subset f^{-1}(\partial(D))\cap C
\subset \partial(f^{-1}(D)\cap C)$.
Ainsi, $f^{-1}(\Delta)\cap\Gamma$ est une cellule de~$\mathscr C'$
qui contient~$x$ et qui est
contenue dans le bord de $f^{-1}(D)\cap C$.
Cela prouve que le bord de toute cellule de~$\mathscr C'$
est réunion de cellules de~$\mathscr C'$.

Cela prouve que $\mathscr C'$ est une décomposition cellulaire de~$X$.

Démontrons qu'elle raffine~$\mathscr C$.
Par construction, toute cellule de~$\mathscr C'$ est une sous-cellule d'une cellule de~$\mathscr C$.
Soit $C$ une cellule de~$\mathscr C$ 
et soit $x$ un point de~$C$; soit $(\Gamma,\Delta)$ l'élément de~$\mathscr E$
associé à~$x$. Alors $\mathring\Gamma$ rencontre~$C$,
donc $\Gamma\subset C$ et $f^{-1}(\Delta)\cap\Gamma$
est une cellule de~$\mathscr C'$ qui contient~$x$ et qui est 
une sous-cellule de~$C$. Ainsi, $C$ est réunion de cellules de~$\mathscr C'$.

Il reste à prouver que le couple $(\mathscr C',\mathscr D')$ est
adapté à~$f$.
Soit $(C,D)$ un élément de~$\mathscr E$;
on a $f(f^{-1}(D)\cap C)=D$ car $D\subset f(C)$.
Inversement, soit $D$ une cellule de~$\mathscr D'$.
Soit $x$ un point de~$f^{-1}(D)$ et soit $(\Gamma,\Delta)$
l'élément de~$\mathscr E$ associé à~$x$.
On a $\Delta\subset D$ et $x\in f^{-1}(\Delta)\cap\Gamma$.
Cela prouve que tout point de~$f^{-1}(D)$ appartient à une cellule de~$\mathscr C'$ qui est contenue dans~$f^{-1}(D)$,
c'est-à-dire que $f^{-1}(D)$ est réunion de cellules de~$\mathscr C'$.

La proposition est ainsi démontrée.
\end{proof}

\begin{rema}\label{rema.decomp-cell-adapte-raffinement}
%
%
Reprenons les notations de la proposition, supposons en outre que 
le couple $(\mathscr C,\mathscr D)$ est adapté à~$f$.
Pour toute décomposition cellulaire~$\mathscr D'$ 
de~$Y$ qui raffine~$\mathscr D$, la proposition
fournit une décomposition cellulaire~$\mathscr C'$
de~$X$ qui raffine~$\mathscr C$ 
et telle que le couple $(\mathscr C',\mathscr D')$ est adapté à~$f$.

Remarquons en outre que, dans ce cas,
la condition $\mathring D\subset f(\mathring C)$ de l'énoncé
équivaut à ce que $f(C)$ soit l'unique cellule de~$\mathscr D$
dont l'intérieur contient~$\mathring D$.
\end{rema}

\begin{coro}\label{prop.decomp-relative}
Soit $X$ et $Y$ des espaces paralinéaires paracompacts et
soit $f\colon  X \to Y$ une application paralinéaire compacte.
Soit $\mathscr C$ une décomposition cellulaire de~$X$,
soit $\mathscr D$ une décomposition cellulaire de~$Y$.
Il existe une décomposition cellulaire~$\mathscr C'$ 
de~$X$ qui raffine~$\mathscr C$ et une décomposition
cellulaire~$\mathscr D'$ de~$Y$ qui raffine~$\mathscr D$
telles que le couple $(\mathscr C',\mathscr D')$ est adapté à
l'application paralinéaire~$f$.
\end{coro}
\begin{proof}
La famille des $f^{-1}(D)$, pour $D\in\mathscr D$, est localement finie.
Commençons ainsi par choisir une décomposition cellulaire~$\mathscr C_1$
de~$X$, plus fine que~$\mathscr C$, et compatible aux applications
paralinéaires $f_D\colon f^{-1}(D)\to D$, pour $D\in\mathscr D$.
Par définition, $f^{-1}(D)$ est réunion de cellules de~$\mathscr C_1$
qui sont des sous-cellules de cellules de~$\mathscr C$,
et pour chaque telle cellule~$C_1$, l'application
$f|_{C_1}\colon C_1\to D$ est affine.

Soit $C_1$ une cellule de~$\mathscr C_1$. 
Si $D$ est une cellule de~$\mathscr D$ telle que $f(C_1)\subset D$,
alors, $f|_{C_1}\colon C_1\to D$ est affine, par construction de~$\mathscr C_1$.
Soit $x$ un point de~$\mathring C_1$
et soit $D$ l'unique cellule de~$\mathscr D$ telle que $f(x)\in \mathring D$.
Comme $f^{-1}(D)$ est réunion de cellules de~$\mathscr C_1$
et comme il rencontre~$\mathring C_1$,
on a $C_1\subset f^{-1}(D)$, donc $f(C_1)\subset D$.

Comme $f$ est compacte, la famille $(f(C_1))_{C_1\in\mathscr C_1}$
est une famille localement finie de parties paralinéaires fermées de~$Y$.
Il existe alors une décomposition cellulaire~$\mathscr D_1$
de~$Y$, raffinant~$\mathscr D$, telle que toute partie de la forme~$f(C_1)$,
pour $C_1\in\mathscr C_1$, soit réunion de cellules.

Le corollaire résulte alors de la proposition,
appliquée aux décompositions cellulaires~$\mathscr C_1$ de~$X$
et $\mathscr D,\mathscr D_1$ de~$Y$.
\end{proof}

\begin{coro}\label{coro.dec-cellulaire-fibre}
Soit $f\colon X\to Y$ une application paralinéaire entre espaces
paralinéaires, soit $\mathscr C$ une décomposition cellulaire de~$X$,
soit $\mathscr D$ une décomposition cellulaire de~$Y$;
on suppose les conditions de la proposition~\ref{prop.relative-raffinement}
satisfaites.
Soit $y\in Y$ ; 
posons $X_y=f^{-1}(y)$ et
notons $\mathscr C_y$ l'ensemble des parties $C\in\mathscr C$
telles que $\mathring C\cap X_y\neq \emptyset$.

\begin{enumerate}
\item
Pour tout $C\in\mathscr C_y$, $C\cap X_y$ est une sous-cellule de~$C$
dont l'intérieur est $\mathring C\cap X_y$,
et la famille $\mathscr D_y:=(C\cap X_y)_{C\in\mathscr C_y}$
est une décomposition cellulaire
de l'espace paralinéaire~$X_y$.

\item
Pour toute cellule~$C$ de~$\mathscr C$, $C\cap X_y$
est une réunion de cellules de la décomposition cellulaire~$\mathscr D_y$. 

\item
Supposons, de plus, que le couple de décompositions cellulaires 
$(\mathscr C,\mathscr D)$
soit adapté à~$f$. Soit $D$ l'unique cellule de~$\mathscr D$
telle que $y\in \mathring D$.
Alors, $\mathscr C_y$ est l'ensemble des cellules~$C$ de~$\mathscr C$
telles que $f(C)=D$.
\end{enumerate}
\end{coro}
\begin{proof}
\begin{enumerate}
\item
Soit $\mathscr D'$ une décomposition cellulaire de~$Y$ qui raffine~$\mathscr D$
et dont $\{y\}$ est une cellule.
Notons $\mathscr C'$ la  famille des $f^{-1}(D)\cap C$, où $C$ est une cellule de~$\mathscr C$ 
et $D$ est une cellule de~$\mathscr D'$ telle que $\mathring D\subset f(\mathring C)$. D'après la proposition~\ref{prop.relative-raffinement},
c'est une décomposition cellulaire de~$X$ qui raffine~$\mathscr C$
et le couple $(\mathscr C',\mathscr D')$ est adapté à~$f$.
Le corollaire résulte alors de la remarque~\ref{rema.relative-restriction}
appliquée au sous-espace paralinéaire $Y'=\{y\}$ de~$Y$.

\item
Procédons par récurrence sur la dimension de~$C$. 
Si $C\in\mathscr C_y$, alors le résultat est vrai.
Sinon, $C\cap X_y$ est contenu dans $\partial(C)\cap X_y$;
par définition d'une décomposition cellulaire, $\partial(C)$ 
est réunion finie de cellules de dimension strictement inférieure à celle de~$C$
et on conclut par récurrence.

\item
Si le couple $(\mathscr C,\mathscr D)$ est adapté à~$f$,
la condition $y\in f(\mathring C)$ équivaut à
$y\in f(C)^\circ$, donc à $f(C)=D$, puisque $f(C)$ est une cellule
de~$\mathscr D$.
\qedhere
\end{enumerate}
\end{proof}

\def\Dom{\operatorname{Dom}}
\def\PL{\mathrm{PL}}
\def\pf{\text{\upshape pf}}
\def\red{\text{\upshape red}}
\def\lf{\text{\upshape lf}}
\def\Bor{\text{\upshape Bor}}
\def\loc{\text{\upshape loc}}
\def\Rad{\text{\upshape Rad}}
\def\Zar{\text{\upshape Zar}}
\def\Tau{{\mathrm T}}
\def\gZ{\mathrm Z}
\def\degcal{\operatorname{dc}\nolimits}
\def\codim{\operatorname{codim}}

\chapter{Espaces G-tropicaux}

Dans ce chapitre, on fixe un sous-corps~$K$ de~$\R$
et un sous-$K$-espace vectoriel non nul~$G$ de~$\R$.
Les espaces paralinéaires que nous considérons sont implicitement
munis d'une $(K,G)$-structure (cf.~\S\ref{ss.KG}).
Par exemple,
lorsque nous parlons de fonction affine sur~$\R^n$,
sa partie linéaire est définie sur~$K$ 
et son terme constant appartient à~$G$.

\section{Espaces G-tropicaux}

Le but de ce qui suit est de définir une notion de structure
d'\emph{espace G-tropical} sur un espace topologique~$X$ ; 
cela va consister en une liste substantielle de données sujettes à un certain 
nombre d'axiomes, que nous allons maintenant 
présenter. 

\subsection{La topologie}
On suppose que $X$ est topologiquement séparé et
localement compact.  On suppose de plus que $X$ possède une base
d'ouverts qui sont dénombrables à l'infini.
\label{G-tropical-topologie}

\subsection{La G-topologie}
On se donne un ensemble $\Dom(X)$ de parties localement fermées
de $X$ que nous appellerons  \emph{domaines de~$X$}.

On dit qu'une partie~$W$ de~$X$ est G-recouverte 
par une partie $\mathscr V$ de~$\Dom(X)$
si chaque élément de~$\mathscr V$ est contenu dans~$W$
et si, pour tout point $x\in W$, il existe
une partie finie~$\mathscr V_x$ de~$\mathscr V$ telle que
$x\in V$ pour tout $V\in\mathscr V_x$ et la réunion des~$V$,
pour $V\in\mathscr V_x$ est un voisinage de~$x$ dans~$W$.

Axiomes sur les domaines : 
\begin{itemize}
\item Toute partie ouverte de~$X$ est un domaine;
\item L'intersection de deux domaines est un domaine;
\item Toute partie localement fermée de~$X$ qui est G-recouverte
par des domaines de~$X$ est encore un domaine de~$X$;
\item Tout domaine de~$X$ est G-recouvert par les domaines compacts
qu'il contient.
\end{itemize}

Il en résulte les faits suivants:
\begin{itemize}
\item Tout recouvrement ouvert d'une partie ouverte est un G-recouvrement;
\item Toute réunion finie de domaines compacts est un domaine compact.
\end{itemize}

La catégorie des domaines de~$X$ (les morphismes étant
les inclusions), munie
de la notion de G-recouvrement,
sera appelée la G-topologie de~$X$.
Plus généralement, tout domaine de~$X$ hérite d'une G-topologie.

\subsection{Les fonctions paralinéaires}
On se donne un G-faisceau $\PL_X$ de fonctions numériques 
continues sur~$X$,
qu'on appelle \emph{fonctions paralinéaires} ;
on demande que pour tout domaine~$W$ de~$X$,
l'ensemble $\PL_X(W)$ soit un sous-espace affine 
de l'espace des fonctions continues sur~$W$.

Plus généralement, si $W$ est un domaine de~$X$ et si $f\colon W\to\R^m$
est une application, on dit que $f$  est paralinéaire
si ses composantes le sont.

Axiomes sur les fonctions paralinéaires.
Soit $W$ un domaine de~$X$ et soit $f\colon W\to\R^m$ une fonction paralinéaire.
\begin{itemize}
\item Si $W$ est compact, alors $f(W)$ est un polyèdre compact de~$\R^m$;
\item Si $P$ est un sous-espace paralinéaire de~$\R^m$
contenant~$f(W)$ et si $\phi\colon P\to\R$ est paralinéaire, alors
$\phi\circ f$ est paralinéaire;
\item Si $P$ est un sous-espace paralinéaire de~$\R^m$,
alors $f^{-1}(P)$ est un domaine de~$W$.
\end{itemize}

En particulier, $f$ induit un morphisme de sites annelés
de~$W$ vers l'espace~$\R^m$ muni de sa G-topologie linéaire par morceaux.

Une \emph{carte G-tropicale} sur $X$ est un triplet $(V,f,P)$ 
où $V$ est un domaine de $X$, où $f$ est une fonction paralinéaire de
$V$ dans un espace affine~$E$
et où $P$ est une partie paralinéaire de $E$ contenant $f(V)$.
En pratique, on pourra supposer que $E$ est l'espace~$\R^m$,
pour un certain~$m$, et la partie~$P$ sera souvent omise de la notation. 

Soit $(V,f)$ une carte G-tropicale et soit $x\in V$.
On appelle \emph{dimension tropicale} de~$f$ en~$x$
la borne inférieure des $\dim(f(W))$, où $W$ est un domaine
compact contenu dans~$V$ qui est voisinage de~$x$.
On la note $\dim_{\trop,x}(f)$.
Il existe un domaine contenu dans~$V$ qui la réalise.

Soit $x$ un point de~$X$.
On appelle $d_\trop(x)$ la borne supérieure des 
dimensions tropicales $\dim_{\trop,f}(f)$,
où $(V,f)$ parcourt l'ensemble des cartes G-tropicales
des domaines~$V$ de~$X$ contenant~$x$.

On note $d_\trop(X)$ la borne supérieure des dimensions
$\dim(f(V))$, où $(V,f)$ parcourt l'ensemble
des cartes G-tropicales sur des domaines de~$X$.
C'est aussi la borne supérieure des $d_\trop(x)$, pour $x\in X$.

\begin{defi}\label{defi.G-tropical}
Un \textit{espace G-tropical}\index{espace G-tropical}
est un espace topologique séparé,  localement compact~$X$,
possédant une base d'ouverts dénombrables à l'infini,
muni des structures suivantes : un ensemble de domaines, une G-topologie
et  un G-faisceau de fonctions paralinéaires,
le tout étant sujet aux axiomes énoncés ci-dessus.
\end{defi}

\subsection{}
Un domaine d'un espace G-tropical a une structure naturelle
d'espace G-tropical.

\subsection{}
Un espace paralinéaire a une structure naturelle d'espace
G-tropical: ses domaines sont ses sous-espaces paralinéaires
et ses fonctions paralinéaires sont celles que l'on croit.

\subsection{}
Un espace $k$-analytique~$X$ a une structure naturelle
d'espace G-tropical : ses domaines sont ses domaines analytiques
et ses fonctions paralinéaires sont les éléments
du plus petit G-faisceau en $\Q$-espaces vectoriels sur~$X$
contenant les fonctions $\log (\abs f)$, où $f$ est une fonction
analytique \emph{inversible} sur un domaine analytique de~$X$.

\subsection{Morphismes}
Soit $X$ et $Y$ des espaces G-tropicaux.
On dit qu'une application $f\colon X\to Y$ est un 
morphisme d'espaces G-tropicaux si les trois conditions suivantes
sont satisfaites:
\begin{enumerate}
\item L'application $f$ est continue ;
\item Pour tout domaine~$V$ de~$Y$, $f^{-1}(V)$ est un domaine de~$X$;
\item Pour tout domaine~$V$ de~$Y$ et toute fonction paralinéaire~$\phi$
sur~$V$, la fonction $\phi\circ f$ sur $f^{-1}(V)$ est paralinéaire.
\end{enumerate}

L'application identique de~$X$ est un morphisme d'espaces G-tropicaux.
Si $f\colon X\to Y$ et $g\colon Y\to Z$ 
sont des morphismes d'espaces G-tropicaux,
leur composée $g\circ f \colon X\to Z$ est un morphisme
d'espaces G-tropicaux.

Soit $V$ un domaine de~$X$. L'injection canonique de~$V$
dans~$X$ est un morphisme d'espaces G-tropicaux.

Soit $V$ un domaine de~$X$ et $f\colon V\to E$ une application paralinéaire.
Alors $f$ est un morphisme d'espaces G-tropicaux.

%


\section{Parties paralinéaires d'un espace G-tropical}

Soit $X$ un espace G-tropical.

\begin{defi}
\label{defi.fidelement-pl}
Soit $P$ une partie localement fermée de~$X$
et soit $f\colon P\to\R^m$ une application. 
On dit que $f$ est une \emph{carte paralinéaire}\index{carte paralinéaire}
sur~$P$ si les propriétés suivantes soient satisfaites : 
\begin{enumerate}
\item il existe un domaine $V$ de $X$ contenant $P$ tel
que $f$ s'étende en une application paralinéaire de $V$ vers $\R^m$ ; 
\item $f$ induit un homéomorphisme de $P$
sur une partie paralinéaire de $\R^m$, dont on note $f^{-1}_P$
la réciproque ; 
\item pour toute carte tropicale $(W,g)$ de $X$,
l'image $f(P\cap W)$ est une partie paralinéaire de $\R^m$, 
et l'application $g\circ f^{-1}_P|_{f(P\cap W)}$ est paralinéaire.
\end{enumerate}

On dit que  $P$ est une partie
\emph{fidèlement paralinéaire}\index{partie fidèlement paralinéaire}
si elle possède une carte paralinéaire.
\end{defi}

\subsection{}
Lorsque $P$ est une partie fidèlement paralinéaire, 
nous noterons $\Lambda(P)$ l'ensemble des cartes paralinéaires sur~$P$.
Un \textit{domaine de définition}\index{carte paralinéaire!domaine de définition} d'une carte paralinéaire~$f$ sera un domaine~$V$ contenant~$P$ 
et muni d'un prolongement paralinéaire de~$f$ 
qui, sauf mention expresse du contraire, sera encore noté~$f$. 
Remarquons que si $P$ est contenu dans un domaine~$W$ de~$X$,
alors $P$~est fidèlement paralinéaire dans~$W$ 
si et seulement si elle l'est dans~$X$. 

\subsection{}
Soit $P$ une partie fidèlement paralinéaire de~$X$.
Toute application $f\in \Lambda(P)$ définit alors par transport
de structure une structure d'espace paralinéaire sur $P$, qui en vertu de la 
condition~c) ne dépend pas de $f$, et que nous 
appellerons sa structure canonique. 
Cette même condition assure que pour toute
carte G-tropicale $(W,g)$, l'intersection $W\cap P$
est une partie paralinéaire de $P$, et que $g|_{W\cap P}$
est paralinéaire. 
Nous considèrerons toujours $P$ comme munie
de sa structure canonique. 

\subsection{}
Soit $P$ une partie fidèlement paralinéaire de~$X$
et soit  $Q$ une partie paralinéaire de $P$.
Alors, $Q$ est une partie
fidèlement paralinéaire de $X$ 
(si $f\in \Lambda(P)$ alors $f|_Q\in \Lambda(Q)$ de façon évidente). 
Réciproquement, si $Q$ est une partie fidèlement paralinéaire de $X$
 contenue dans $P$ et si
$f\in \Lambda(P)$ alors $f|_Q$ est paralinéaire
(parce que $Q$ est fidèlement paralinéaire) et induit un homéomorphisme
de $Q$ sur son image (parce que c'est déjà le cas au niveau de $P$)
si bien que $f(Q)$ est paralinéaire ; par conséquent, $Q$
est une partie paralinéaire de $P$. 

\begin{defi}\label{definition-partie-paralineaireX}
Soit $\Sigma$ une partie localement fermée de $X$. 
Nous dirons que $\Sigma$ est \textit{paralinéaire}
si elle est G-recouverte par des parties fidèlement paralinéaires de~$X$ de la forme $\Sigma\cap W$, où $W$ est un domaine de~$X$. 
\end{defi}

Une partie fidèlement paralinéaire est paralinéaire. 

\subsection{}
Soit $\Sigma$ une partie paralinéaire de $X$. 
Soit $(P_i)$ un G-recouvrement de~$\Sigma$ 
par des parties fidèlement paralinéaires, chacune d'elles étant de la forme
$W_i\cap \Sigma$ où $W_i$ est un domaine de $X$ ; 
choisissons pour tout~$i$ une application $f_i\in \Lambda(P_i)$, de but $\R^{n_i}$. 
D'après la condition~c) de la définition d'une partie fidèlement
paralinéaire de $X$
la famille $(P_i,f_i)$ est un atlas paralinéaire sur~$\Sigma$, 
et la classe d'équivalence de cet atlas est indépendante du choix des $P_i$ et des
$f_i$. La partie paralinéaire~$\Sigma$ 
hérite ainsi d'une structure paralinéaire, dite canonique,
dont nous la munirons toujours implicitement (qui coïncide avec celle
déjà définie lorsque $\Sigma$ est fidèlement paralinéaire). 

\begin{rema}
Soit $\Sigma$ une partie paralinéaire de~$X$.
Alors $\Sigma$ est G-recouverte par des parties fidèlement paralinéaires
compactes de la forme $\Sigma\cap W$, où $W$ est un domaine compact de $X$. 

En effet, choisissons pour commencer un G-recouvrement $(P_i)$ de $\Sigma$ par des parties
fidèlement paralinéaires, chacune des $P_i$ étant de la forme $\Sigma\cap W_i$ où $W_i$
est un domaine de $X$. Et choisissons pour tout $i$ une application 
$f_i\colon P_i\to \R^{n_i}$ appartenant à $\Lambda(P_i)$. Quitte à remplacer $W_i$ par son intersection
avec un domaine de définition de $f_i$, on peut supposer que $f_i$ est définie sur $W_i$. 
Pour tout~$i$, on fixe un G-recouvrement $(Q_{ij})$ de $f_i(P_i)$ par des parties paralinéaires compactes, et 
l'on note $W_{ij}$ le domaine $f_i^{-1}(Q_{ij})$ de $W_i$. L'intersection 
$R_ij$ de $W_{ij}$ et $\Sigma$ est alors égale
à  $(f_i|_{P_i})^{-1}(Q_{ij})$ ; c'est une partie paralinéaire compacte de $P_i$, et en particulier une partie
fidèlement paralinéaire de $X$. Chacun des $R_{ij}$ possède un voisinage $W'_{ij}$ dans $W_{ij}$
qui est un domaine compact, et $R_{ij}=\Sigma \cap W'_{ij}$. La famille des $(R_{ij})$ constitue alors
un G-recouvrement de $\Sigma$ de la forme cherchée. 
\end{rema}

\begin{lemm}\label{lemme-proprietes-paralineairesX}
Soit $\Sigma$ une partie paralinéaire de $X$, munie de sa structure paralinéaire canonique. 

\begin{enumerate}
\item Soit $P$ une partie fidèlement paralinéaire de 
$X$ contenue dans $\Sigma$. C'est aors une partie paralinéaire de $\Sigma$, 
et sa structure paralinéaire héritée de celle de $\Sigma$ coïncide avec sa structure canonique. 

\item Soit $P$ et $S$ deux parties fidèlement paralinéaires de $X$ contenues dans $\Sigma$. Leur intersection
est paralinéaire dans chacune d'entre elles. 

\item
Soit $W$ un domaine de $X$ et soit $g$ une fonction paralinéaire sur $W$.
L'intersection $\Sigma\cap W$ est un sous-espace paralinéaire de $\Sigma$, 
et $g|_{\Sigma\cap W}$ est paralinéaire.

\item
Soit $\Tau$ une partie localement fermée de $X$ telle que $\Tau\subset\Sigma$.
Alors $\Tau$ est une partie paralinéaire de $\Sigma$ si et seulement si c'est une partie paralinéaire de $X$. 

\item Soit $\Tau$ une partie paralinéaire compacte de $\Sigma$. 
Il existe un domaine compact~$W$ de~$X$ tel que $W\cap \Sigma=\Tau$. 
\item Soit $x$ un point de $\Sigma$ et soit $\Tau$ un voisinage paralinéaire
compact de $x$ dans $\Sigma$. Il existe un voisinage domanial compact $W$ de $x$ dans $X$
tel que $W\cap \Sigma=\Tau$.

\end{enumerate}
\end{lemm}
\begin{proof}
\begin{enumerate}
\item  Soit $(P_i)$ un G-recouvrement de $\Sigma$ par des parties fidèlement paralinéaires
telles que chacune des $P_i$ soit égale à $\Sigma\cap W_i$
pour un certain domaine $W_i$ de $X$ ; pour tout $i$, choisissons un élément $f_i\colon
P_i\to \R^{i}$ de $\Lambda(P_i)$. 

On a alors $P\cap P_i=P\cap W_i$, et $P\cap P_i$ est dès lors une partie
paralinéaire de $P$ en restriction à laquelle $f_i$ est paralinéaire. 
Par conséquent $f_i(P\cap (P_i))$ est une partie paralinéaire de
$\R^{n_i}$, et $P\cap P_i$ est donc également une partie paralinéaire de $P_i$, et \textit{a fortiori}
de $\Sigma$ ; notons que les structures paralinéaires sur $P\cap P_i$ héritées
de celle de $P_i$ et de celles de $P$ coïncident, étant toutes deux induites
par l'homéomorphisme $f_i|_{P\cap P_i}\colon P\cap P_i\simeq f_i(P\cap P_i)$.
Puisque les $P_i\cap P$ constitue un G-recouvement de $P$, il s'ensuit que $P$ est une partie paralinéaire
de $\Sigma$, et que sa structure paralinéaire héritée de celle de $\Sigma$ coïncide avec sa structure
paralinéaire canonique. 

\item C'est une conséquence immédiate de a). 

\item
Cela se vérifie G-localement, ce qui permet de supposer que $\Sigma$ est fidèlement
paralinéaire, auquel cas cela découle de la condition c) de la définition d'une partie
fidèlement paralinéaire.

\item
Supposons que $\Tau$ soit une partie paralinéaire de $\Sigma$. 
Choisissons un G-recouvrement $(\Sigma \cap V_i)$ de $\Sigma$
par des parties fidèlement paralinéaires de $X$,
où les~$V_i$ sont des domaines de~$X$. 
Chacun des $T\cap V_i$ est une partie paralinéaire de~$\Sigma\cap V_i$, 
donc une partie fidèlement paralinéaire de $X$, 
et les $\Tau\cap V_i$ constituent un G-recouvrement de $\Tau$, 
qui est dès lors une partie paralinéaire de $X$.

Réciproquement, supposons que $\Tau$ soit une partie paralinéaire de $X$. 
Elle est alors par définition G-recouverte 
par des parties fidèlement paralinéaires de $X$, 
qui sont des parties paralinéaires de $\Sigma$ d'après a). C'est donc une partie
paralinéaire de $\Tau$. 

\item Soit $\Tau$ une partie paralinéaire compacte de $\Sigma$. Choisissons $(P_i,W_i, f_i)$
comme dans la preuve de a), en supposant de plus que les $P_i$ et les $W_i$ sont compacts, 
ce qui est toujours possible comme expliqué à la fin de \ref{definition-partie-paralineaireX}. 
Puisque $\Tau$ est compact, il existe un ensemble fini $I$ d'indices tel que $\Tau\subset \bigcup_{i\in I}
P_i$. Si $i\in I$, l'intersection $\Tau\cap P_i$ est paralinéaire, si bien que $f_i(\Tau\cap P_i)$
est une partie compacte de $\R^{n_i}$. L'image réciproque $f_i^{-1}(f_i(\Tau\cap P_i))$ est un domaine
compact $W'_i$ de $W_i$, dont la trace sur $\Sigma$ est égale à $\Tau\cap P_i$. La réunion 
$\bigcup_{i\in I}W'_i$ est alors un domaine compact de $X$ dont l'intersection avec $\Sigma$
est égale à $\Tau$. 

\item Soit $\Tau$ un voisinage paraliné	aire compact de $x$. Par l'assertion (e) que l'on
vient d'établir il existe un domaine compact $W$ de $X$ tel que $W\cap \Sigma=\Tau$. Par ailleurs puisque $\Tau$
est un voisinage de $x$ dans $\Sigma$ il existe un voisinage ouvert $U$ de $x$ dans $X$ tel que $U\cap \Sigma
\subset \Tau$. Soit $U'$un voisinage domanial compact de $x$ dans$U$. La réunion $W\cap U'$ est alors un voisinage domanial compact
de $x$ dans$X$ dont la trace sur $\Sigma$ est égale à $\Tau$. 
\qedhere
\end{enumerate}
\end{proof}

\begin{lemm}[Parties paralinéaires et G-topologie]
\label{lemme-partiespl-glocal}
%
Soit $\Sigma$ une partie localement fermée de $X$. 
Supposons qu'il existe une famille $(W_i)$ de
domaines de $X$ qui G-recouvrent $\Sigma$ et 
tels que pour tout~$i$, l'intersection $\Sigma\cap W_i$
soit paralinéaire. La partie~$\Sigma$ est alors paralinéaire. 
\end{lemm}

\begin{proof}
Par définition, chacun des $\Sigma\cap W_i$ est G-recouvert
par des parties fidèlement paralinéaires de~$X$ de la forme
$\Sigma\cap W_{i,j}$, où les $W_{i,j}$ sont des domaines de~$W_i$.
Par suite, $\Sigma$ elle-même est G-recouverte par les 
parties fidèlement paralinéaires $\Sigma\cap W_{i,j}$ de~$X$, 
et est dès lors paralinéaire. 
\end{proof}

\begin{lemm}\label{lemm.paralineaire-Glocal}
Soit $(\Sigma_i)$ une famille localement
finie de parties paralinéaires fermées de $X$. 
Les assertions suivantes sont équivalentes : 
\begin{enumerate}\def\theenumi{\roman{enumi}}\def\labelenumi{(\theenumi)}
\item La réunion $\Sigma$ des $\Sigma_i$ est une partie paralinéaire de $X$. 
\item Il existe une famille $(W_j)$ de domaines compacts de $X$ qui G-recouvrent
$\Sigma$ et, pour tout $j$, une partie paralinéaire $T_j$
de $W_j$ contenant $\Sigma\cap W_j$. 
\end{enumerate}

Lorsqu'elles sont satisfaites,
$\Sigma_i\cap \Sigma_j$
est pour tout $(i,j)$ une partie paralinéaire de $\Sigma$, 
et donc de $\Sigma_i$ et $\Sigma_j$. 
\end{lemm}
\begin{proof}
Supposons que (i) soit vraie. Alors (ii) est vraie en prenant pour $(W_j)$
un G-recouvrement (arbitraire)
de $X$ par des domaines compacts, et en posant $T_j=\Sigma\cap W_j$ pour tout $j$. 

Supposons que (ii) soit vraie. Pour tout $j$, l'intersection
$\Sigma_\cap W_j$ s'écrit $\bigcup_{i\in I}\Sigma_i\cap W_j$ 
où $I$ est un ensemble fini d'indices (par compacité de
$W_j$ et locale finitude de la famille des $\Sigma_i$). 
Chacun des $\Sigma_i\cap W_j$ est paralinéaire, et est donc une partie
paralinéaire fermée de $T_j$. Il en résulte que $\Sigma\cap W_j$ est une partie
paralinéaire de $T_j$, et donc de $X$. On déduit alors
du lemme \ref{lemme-partiespl-glocal}
que $\Sigma$ est paralinéaire. 
\end{proof}

\begin{lemm}\label{paracomp.ouvert}
Soit $X$ un espace G-tropical paracompact 
(resp.\ dénombrable à l'infini) 
et soit $\Sigma$ une partie paralinéaire fermée de~$X$.
Pour tout ouvert~$U$ de $X$,
$\Sigma\cap U$ possède un voisinage ouvert dans~$U$
qui est paracompact (resp.\ dénombrable à l'infini).
\end{lemm}
\begin{proof}
En raisonnant composante connexe par composante connexe, 
on peut supposer que $X$ est connexe et donc dénombrable à l'infini. 
Il en est donc de même de sa partie fermée~$\Sigma$,
ainsi que de $\Sigma \cap U$  d'après le lemme~\ref{lemm.pl-paracompact}.
L'assertion découle donc du lemme~\ref{comp.fort.paracomp}.
\end{proof}

\subsection{}
Soit $\Sigma$ une partie paralinéaire de~$X$.
L'injection de~$\Sigma$ dans~$X$
est un morphisme d'espaces G-tropicaux.

\begin{prop}
Soit $p\colon Y \to X$ un morphisme d'espaces G-tropicaux.
Soit $\Sigma$ une partie paralinéaire de~$X$
et soit $\Tau$ une partie paralinéaire
de~$Y$ telle que $p(\Tau)\subset\Sigma$. Alors l'application 
$p|_\Tau\colon\Tau\to\Sigma$ est paralinéaire.
\end{prop}
\begin{proof}
L'assertion est G-locale sur~$X$ et~$Y$. Par définition d'une partie
paralinéaire, on peut donc  supposer que $\Sigma$ est
fidèlement paralinéaires.
Soit $f\in\Lambda(\Sigma)$. Comme $p$ est un morphisme
d'espaces G-tropicaux, $f\circ p$ est une application paralinéaire sur~$Y$.
Par suite, $f\circ p|_\Tau$ est paralinéaire (lemme~\ref{lemme-proprietes-paralineairesX}).
Comme $f$ identifie~$\Sigma$ à une partie paralinéaire d'un espace
numérique, cela démontre que $p$ induit une application paralinéaire
de~$\Tau$ vers~$\Sigma$.
\end{proof}

\section{Formes de type~$(p,q)$ sur un espace G-tropical}

Dans ce paragraphe, on fixe une classe de coefficients~$\mathscr G$.

\subsection{}\label{sss.carte-tropicale}
Soit $X$ un espace G-tropical.

Les cartes G-tropicales sur les domaines de~$X$ 
forment une catégorie essentiellement petite,
un morphisme d'une carte G-tropicale~$(f\colon X\ra E,P)$ dans une 
autre $(f'\colon X\ra E',P')$
étant une application affine $q\colon E\ra E'$ tel
que $f'=q\circ f$ et $q(P)\subset P'$.

Cette catégorie a un objet final pour lequel 
$E=\{0\}$ et $P$ est un point.
Elle est filtrante:
si $(f\colon X\to E,P)$
et $(f'\colon X\to E',P')$ sont deux cartes G-tropicales sur~$X$,
alors $(f\times f'\colon X\to E\times E', P\times P')$
est une carte G-tropicale sur~$X$ qui domine les deux
cartes G-tropicales données.

\begin{rema}
Soit $X$ un espace G-tropical compact et
soit $(f\colon X\to E,P)$ une G-carte tropicale sur~$X$.
\begin{enumerate}
\item L'ensemble $f(X)$
est un sous-espace paralinéaire compact de~$E$ contenu dans~$P$
(théorème.~\ref{prop.trop.global})
et la carte donnée se factorise par $(f\colon X\to E,f(X))$.
On dira parfois, par abus, que $f$ est une carte G-tropicale sur~$X$.
\item Pour toute décomposition cellulaire~$(P_i)$ de~$f(X)$,
la famille $(f^{-1}(P_i))$ est un G-recouvrement fini de~$X$
par des domaines compacts dont l'image par~$f$ est une cellule.
\end{enumerate}
\end{rema}

\subsection{}\label{sss.Gpq-pf}
Pour tout domaine~$V$ de~$X$ 
et tout couple~$(p,q)$ d'entiers naturels, notons
$\mathscr G_{\pf}^{p,q}(V)$ la  limite inductive des $\R$-espaces
vectoriels $\mathscr G^{p,q}_P(P)$, lorsque
$(f\colon V\to E, P)$ parcourt l'ensemble des cartes G-tropicales sur~$V$.

Lorsque $V$ parcourt l'ensemble des domaines de~$X$,
les espaces vectoriels~$\mathscr G_{\pf}^{p,q}(V)$ définissent 
un \emph{préfaisceau} sur le site~$X_\groth$.

Soit $V$ un domaine de~$X$ 
et soit $(f\colon V\to E,P)$ une carte G-tropicale sur~$V$. 
Toute forme $\alpha$ appartenant à $\mathscr G^{p,q}_P(P)$ 
définit un élément
de $\mathscr G^{p,q}_{\pf}(V)$ qui est noté $f^*\alpha$.

Par définition, toute section de
$\mathscr G^{p,q}_\pf(V)$ est de cette forme;
on peut même supposer que $P$ est ouvert.
Plus précisément, soit $(f\colon V\to E,P)$ une carte
G-tropicale sur~$V$, soit $\alpha\in\mathscr G^{p,q}_P(P)$;
il existe un ouvert~$U$ de~$E$ dont $P$ est une partie fermée;
d'après le lemme~\ref{lemm.extensions}, il existe
$\beta\in \mathscr G^{p,q}_{E}(U)$ telle que $\alpha=\beta|_P$;
on a alors $f^*\alpha=f^*\beta$.


\begin{defi}
On note $\mathscr G_{X_\groth}^{p,q}$ le faisceau sur~$X_\groth$
associé au préfaisceau~$\mathscr G_\pf^{p,q}$.
Ses sections 
sont appelées \emph{G-formes de type~$(p,q)$ à coefficients dans~$\mathscr G$}%
\index{forme!G-forme sur un espace G-tropical}.
 
Lorsque la classe~$\mathscr G$ est la classe~$\mathscr A$
des fonctions lisses, on parle de \emph{G-forme lisse de type~$(p,q)$}.
\index{forme!forme G-lisse sur un espace G-tropical}
\end{defi}

Si $X$ est l'espace G-tropical associé à un espace paralinéaire,
on retrouve ainsi le faisceau des G-formes de type~$(p,q)$
à coefficients dans~$\mathscr G$.

\subsection{}\label{sss.support}
%
On définit le \emph{support}\index{support d'une G-forme}
d'une G-forme de type~$(p,q)$ sur~$X$
comme le complémentaire du plus grand ouvert de~$X$ en
restriction auquel cette forme est nulle. 


Si $(X_i)$ est un G-recouvrement de~$X$, le support d'une G-forme~$\omega$
coïncide avec la réunion des supports des G-formes $\omega|_{X_i}$.
Il est tout d'abord évident 
que $\supp(\omega|_{X_i})\subset \supp(\omega)\cap X_i$
pour tout~$i$. Inversement, soit $x$ un point de~$X$
qui n'appartient au support d'aucune des formes $\omega|_{X_i}$;
démontrons que $x$ n'appartient pas au support de~$\omega$.
Soit $I$ un ensemble fini d'indices tel que $x\in \bigcap_{i\in I}X_i$
et $\bigcup_{i\in I}X_i$ soit un voisinage de~$x$.  
Pour tout~$i\in I$, soit $V_i$ un voisinage de~$x$ dans~$X_i$
qui est un domaine compact
et sur lequel $\omega|_{V_i}=0$; soit $V=\bigcup_{i\in I}V_i$;
c'est un domaine compact de~$X$ et un voisinage de~$x$.
Par construction, la G-forme $\omega|_V$ est G-localement nulle, donc nulle.

\begin{prop}\label{lemm.support-comp-ferme}
Soit $X$ un espace G-tropical et soit $\omega$ une G-forme
de type~$(p,q)$ sur~$X$.
Le support de~$\omega$ est contenu dans l'ensemble des points~$x\in X$
tels que $d_\trop(x)\geq \max(p,q)$.
\end{prop}
\begin{proof}
Soit $x\in X$ un point tel que $d_k(x)<\max(p,q)$;
démontrons que $\omega$ est nulle au voisinage de~$x$.
Par hypothèse, il existe des domaines compacts
$V_1,\dots,V_n$ de~$X$ dont la réunion est un voisinage~$V$ de~$x$
dans~$X$, des cartes G-tropicales $f_i\colon V_i\to E_i$ et des 
formes $\alpha_i$ sur $f_{i,\trop}(V_i)$ 
tels que $\omega|_{V_i}=f_i^*\alpha_i$ pour tout~$i$.
La dimension tropicale de~$f_i$ en~$x$ est majorée
par~$d_\trop(x)$; quitte à diminuer les~$V_i$,
on peut supposer que $\dim(f_{i}(V_i))\leq d_\trop(x)$;
alors, $\alpha_i=0$. On a donc $\omega|_V=0$.
\end{proof}

\begin{coro}\label{coro.faisceaux-nuls}
Si $d_\trop(X)<\max(p,q)$, alors
les faisceaux~$\mathscr F_X^{p,q}$  et $\mathscr F_{X_{\mathrm G}}^{p,q}$
sont nuls.
\end{coro}

\subsection{}
Soit $\omega$ une G-forme (resp. une forme) 
de type~$(p,q)$ à coefficients dans~$\mathscr G$ sur~$X$.
Si $V$ est un domaine de~$X$,
une \emph{G-tropicalisation} de $\omega$ sur~$V$ est la donnée d'une 
carte G-tropicale $(f\colon V\ra E,P)$ et
d'une forme $\alpha\in \mathscr G^{p,q}_P(P)$
telle que $\omega=f^*\alpha$. 

On dira que $\omega$ est G-tropicale 
si elle admet une G-tropicalisation sur~$X$.

Si $V$ est un domaine de~$X$ tel que $\omega$
possède une G-tropicalisation sur~$V$, 
nous dirons aussi que~$V$ G-tropicalise $\omega$. 

Soit $V$ un domaine compact de~$X$ qui G-tropicalise $\omega$,
et soit $(f V\to E, P,\alpha)$ une G-tropicalisation de $\omega|_V$. 
Comme $f(V)$ est un sous-espace paralinéaire de~$E$, 
$(f\colon  V\to E,f(V),\alpha|_{f(V)})$ est encore 
une G-tropicalisation de $\omega|_V$ ; une telle 
G-tropicalisation sera simplement notée
$(f\colon V\to E,\alpha)$ : 
\emph{lorsque~$V$ est compact,
l'omission du sous-espace paralinéaire~$P$
signifie qu'il est égal à $f(V)$.}

\begin{rema}
Soit $\mathscr H$ une classe de coefficients contenant~$\mathscr G$.
Observons que les G-formes 
à coefficients dans~$\mathscr G$ s'identifient à des G-formes
à coefficients dans~$\mathscr H$.

En effet, le G-préfaisceau $\mathscr G^{p,q}_\pf$ sur~$X_\groth$ s'identifie
à un sous-G-préfaisceau de $\mathscr H^{p,q}_\pf$,
de sorte que 
le G-faisceau $\mathscr G^{p,q}_{X_\groth}$ 
sur~$X$
s'identifie alors à un sous-G-faisceau 
de $\mathscr H^{p,q}_{X_\groth}$.
Ces identifications sont compatibles avec les constructions
faites dans ce paragraphe.
\end{rema}

\subsection{}
Le faisceau
$\mathscr G_{X_{\mathrm G}}:=\mathscr G^{0,0}_{X_{\mathrm G}}$
est un faisceau en $\R$-algèbres
et les faisceaux $\mathscr G_{X_\mathrm G}^{p,q}$ 
sont des $\mathscr G_{X_{\mathrm G}}$-modules. 
Leur somme directe  $\mathscr G_{X_{\mathrm G}}^{*,*} = \bigoplus_{p,q}\mathscr G_{X_{\mathrm G}}^{p,q}$
est un faisceau en algèbres bigraduées commutatives:
si $\omega$ est de type~$(p,q)$ et $\omega'$ de type~$(p',q')$, alors
\[ \omega' \wedge \omega = (-1)^{(p+p')(q+q')} \omega\wedge\omega'. \]

 
\subsection{}
Le G-préfaisceau~$\mathscr G^{*,*}_{\pf}$
est naturellement muni d'un opérateur~$\mathrm J$
qui s'étend au faisceau $\mathscr G^{*,*}_{X_\groth}$. 
Ce sont des morphismes d'algèbres.
On dit qu'une section~$\omega$ de $\mathscr G^{p,q}_{X_\groth}$ 
est symétrique si $p=q$ et si $\mathrm J\omega=\omega$.

Soit $\mathscr H$ une classe de coefficients telles que les dérivées
partielles des fonctions de~$\mathscr G$ 
existent et appartiennent à~$\mathscr H$. 
On dispose alors d'opérateurs différentiels~$\di$ et $\dc$ 
de $\mathscr G^{*,*}_\pf$ dans~$\mathscr H^{*,*}_\pf$;
ils induisent des opérateurs différentiels 
de $\mathscr G^{*,*}_{X_\groth}$ dans~$\mathscr H^{*,*}_{X_\groth}$,
caractérisés par les relations
\[ \di f^*\alpha=f^*\di\alpha, \quad \dc f^*\alpha=f^*\dc\alpha, \]
pour toute carte G-tropicale $(f\colon U\ra E,P)$ et toute
forme~$\alpha$ de type~$(p,q)$ à coefficients dans~$\mathscr G$ sur~$P$.
Les différentielles vérifient la règle de Leibniz:
si $\omega$ et~$\omega'$ sont respectivement de type~$(p,q)$
et~$(p',q')$, on a
\begin{equation}
 \di (\omega\wedge \omega')
=\di\omega\wedge\omega'+(-1)^{p+q}\omega\wedge\di\omega', \end{equation}
et de même pour $\dc(\omega\wedge\omega')$.

Si $\mathscr G\subset \mathscr C^2$, ces opérateurs vérifient
$\di\di=\dc\dc=0$.

 
\subsection{}
Cette construction est fonctorielle. Plus précisément,
si $f\colon Y\to X$ est un morphisme
d'espaces G-tropicaux, on dispose de morphismes de faisceaux
\[  f^*\colon \mathscr G^{*,*}_{X_{\mathrm G}}
\to f_* (\mathscr G^{*,*}_{Y_{\mathrm G}}) \]
sur $X_\groth$ qui sont des morphismes d'anneaux gradués,
commutent
aux dérivations~$\di$ et $\dc$ (lorsqu'elles existent) et
à l'opérateur~$\mathrm J$.

Lorsque le contexte rend évident le morphisme~$f$,
on se permettra parfois de noter 
$\omega|_Y$
l'image 
dans $\mathscr G^{p,q}_{Y_{\mathrm G}}(Y)$
d'une G-forme $\omega\in\mathscr G^{p,q}_{X_{\mathrm G}}$.

\subsection{}
Soit $(p,q)$ un couple d'entiers naturels.
Soit $\Sigma$ une partie paralinéaire de~$X$
et soit $j$ l'inclusion de~$\Sigma$ dans~$X$.
C'est un morphisme d'espaces G-tropicaux.
On dispose donc d'un morphisme de faisceaux 
de $\mathscr F^{p,q}_{X_\groth}$
dans~$j_*\mathscr F^{p,q}_{\Sigma_G}$.
On notera $\omega\mapsto \omega|_\Sigma$ l'image d'une forme~$\omega$
et on dira que c'en est la \emph{restriction} à~$\Sigma$.

Concrètement, cette restriction se calcule de la façon suivante.

Soit $V$ un domaine de~$X$, soit $(f\colon V\to E,P)$
une carte G-tropicale sur~$V$,
 et soit $\alpha\in\mathscr F^{p,q}_P(P)$.
Alors $\Sigma\cap V$ est une partie paralinéaire de~$V$,
l'application $f\colon \Sigma\cap V\to E$
est paralinéaire et son image est contenue dans~$P$,
l'image réciproque $f^*\alpha$ de~$\alpha$ 
est une forme de type~$(p,q)$
sur l'espace paralinéaire~$\Sigma\cap V$.
Par définition de l'ensemble $\mathscr F^{p,q}_\pf(V)$
comme limite inductive, la forme $f^*\alpha$
ne dépend que de la classe de~$f^*\alpha$ dans cet ensemble.
Cette construction définit un morphisme de préfaisceaux
de $\mathscr F^{p,q}_{X,\pf}$ dans $j_*\mathscr F^{p,q}_{\Sigma}$.

Par passage aux faisceaux associés, on en déduit
le morphisme 
de restriction de $\mathscr F^{p,q}_{X_\groth}$
dans~$j_*\mathscr F^{p,q}_{\Sigma_G}$.

\subsection{}\label{sss.forme-volume-restriction}
Soit $f\colon Y\to X$ un morphisme d'espaces G-tropicaux,
soit $\Tau$ une partie paralinéaires de~$Y$ et soit $\Sigma$
une partie  paralinéaire de~$X$ telle que $f(\Sigma)\subset \Tau$.
Pour toute forme $\omega\in\mathscr F^{p,q}_{X_\groth}(X)$,
on a $(f^*\omega)|_\Tau = (f|_\Tau)^*(\omega|_\Sigma)$.

Nous appliquerons notamment cette remarque 
lorsque $Y$ est un domaine de~$X$ et $\Tau=\Sigma\cap Y$,
ou bien lorsque $Y=X$ et $\Tau$ est une partie polyédrale de~$\Sigma$.

\section{Nullité des G-formes}

\begin{prop}\label{prop.controle-tige}\label{prop.separation.prefaisceau}
Soit~$X$ un espace G-tropical compact,
soit~$f\colon  X\to E$ une carte G-tropicale,
soit~$P$ un sous-ensemble paralinéaire de~$E$ 
contenant~$f(X)$ et soit~$\alpha\in \mathscr G^{p,q}_P(P)$. 
Les assertions suivantes sont équivalentes : 
\begin{enumerate}
\item La restriction de~$\alpha$ au sous-espace paralinéaire~$f(X)$ est nulle;
\item La section globale $f^*\alpha$ de~$\mathscr G^{p,q}_\pf$ est nulle.
\end{enumerate}
\end{prop}
\begin{proof}
L'hypothèse~(1)  entraîne que l'image de~$\alpha$
dans la carte tropicale $(f,f(X))$ est nulle, de sorte que (2)
vaut.
Supposons que~(2) est vérifiée et prouvons~(1). 

On peut remplacer $P$ par~$f(X)$.
Par définition, il existe une carte G-tropicale $g\colon X\to F$, 
une morphisme affine $p\colon F\to E$ tel que $p\circ g=f$,
et un sous-espace paralinéaire~$Q$ de~$F$ contenant $g(X)$
tels que $p(Q)\subset P$
et 
$p^*\alpha=0$ dans~$\mathscr G^{p,q}_Q(Q)$.
On peut aussi remplacer~$Q$ par $g(X)$;
on a alors $p(Q)=P$.

Le morphisme de~$Q$ sur~$P$ induit par $p$ possède
G-localement des sections.
D'après le corollaire~\ref{coro.image-polyedre-sections},
il existe une décomposition cellulaire~$(C_i)$ de~$P$
et, pour tout~$i$, une section de~$p_\trop$ sur~$C_i$
à valeurs dans~$Q$.
Puisque $p^*\alpha=0$,
on a donc $\alpha|_{C_i}=0$ pour tout~$i$. 
Par suite, $\alpha=0$.
\end{proof}

\begin{coro}\label{coro.separation.prefaisceau}
Soit $X$ un espace G-tropical compact,
soit $\omega\in\mathscr G_{\pf}^{p,q}(X)$
et soit $(X_i)$ un G-recouvrement de~$X$.
Supposons que pour tout~$i$, la restriction de~$\omega$ à~$X_i$
soit nulle; alors $\omega=0$.
\end{coro}
\begin{proof}
Tout G-recouvrement admet un G-recouvrement
plus fin formé de domaines compacts.
Comme $X$ est compact, cela permet de supposer que le recouvrement~$(X_i)$
est formé d'un nombre fini de domaines compacts.
Par définition du préfaisceau~$\mathscr G_{\pf}^{p,q}$,
il existe une carte G-tropicale  $f\colon X\to E$,  un sous-espace
paralinéaire~$P$ de~$E$ contenant $f(X)$,
et une forme $\alpha\in\mathscr G^{p,q}_P(P)$ 
tels que $\omega=f^*\alpha$. 

Il résulte que de la proposition~\ref{prop.separation.prefaisceau}
que $\alpha|_{f(X_i)}=0$ pour tout~$i$.
Comme $f(X)$  est la réunion
des $f(X_i)$, la forme $\alpha$ est nulle sur $f(X)$
(lemme~\ref{lemm.PL-separation-prefaisceau-plonge}).
On a donc $\omega=f^*\alpha=0$.
\end{proof}

\begin{rema}
L'hypothèse que $X$ est compact est cruciale 
pour le corollaire~\ref{coro.separation.prefaisceau}.

Prenons pour espace G-tropical l'ensemble~$X=\R$, muni de la topologie
discrète, et pour carte l'application identique $f\colon\R\to\R$.
Soit $\alpha$ la forme $\di x$ sur~$\R$.
Soit $g\colon X\to\R^n$ une carte qui domine la carte~$f$,
et soit $p\colon\R^n\to\R$ une application affine telle que $f=p\circ g$.
Alors, $p$ est surjective, donc $p^*\alpha\neq 0$. Cela prouve
que la section~$f^*\alpha$ de $\mathscr G^{1,0}_\pf$ n'est pas nulle.
En revanche, pour tout~$t\in\R$, le singleton~$V_t=\{t\}$ est un domaine
de~$X$ de tropicalisation $P_t=\{t\}$ et $\alpha|_{P_t}=0$,
si bien que $\alpha|_{V_t}=0$.
Comme ces domaines~$V_t$ forment un G-recouvrement de~$X$,
cela montre que le préfaisceau $\mathscr G_{\pf}^{1,0}$
n'est pas séparé.
%
%
\end{rema}

\begin{prop}\label{prop.nullite-formes}
Soit $X$ un espace G-tropical, soit $f\colon X\to E$ une
carte G-tropicale,
soit $P$ un sous-espace paralinéaire de~$E$ 
contenant~$f(X)$ et soit $\alpha\in\mathscr F^{p,q}_P(P)$.

Les propriétés suivantes sont équivalentes:
\begin{enumerate}
\item Pour tout domaine compact~$V$ de~$X$,
la restriction de~$\alpha$ au polytope~$f(V)$  est nulle;
\item Pour tout domaine compact~$V$ de~$X$,
la restriction de~$f^*\alpha$ à~$V$ est nulle;
\item L'image de~$f^*\alpha$ dans~$\mathscr F^{p,q}_{X_{\mathrm G}}(X)$ 
est nulle.
\end{enumerate}
\end{prop}
\begin{proof}
D'après la proposition~\ref{prop.separation.prefaisceau},
les assertions~(1) et~(2) sont équivalentes.

Supposons (1). Comme $X$ admet un G-recouvrement par des domaines
compacts, l'assertion~(3) s'en déduit.

Supposons~(3). Il existe donc un G-recouvrement~$(X_i)$
tel que $f^*\alpha|_{X_i}=0$ en tant que section 
du préfaisceau~$\mathscr F^{p,q}_\pf$.
Soit $V$ un domaine compact de~$X$.
La restriction de~$f^*\alpha$ à~$X_i\cap V$ est nulle pour tout~$i$;
comme la famille $(X_i\cap V)$ est un G-recouvrement de~$V$,
le corollaire~\ref{coro.separation.prefaisceau}
entraîne que $f^*\alpha|_V=0$, d'où l'assertion~(2).
Cela termine la preuve de la proposition.
\end{proof}

\begin{coro}\label{coro.support} 
Soit~$X$ un espace G-tropical compact, soit~$f\colon X\to E$ 
une carte G-tropical
et soit~$P$ le sous-espace paralinéaire~$f(X)$. 
Soit~$\alpha\in \mathscr G_{P}^{p,q}(P)$, 
et soit~$\omega$ la G-forme~$f^*\alpha$. 
Notons~$K$ le support de~$\omega$ dans~$X$ 
et~$H$ celui de~$\alpha$ dans~$f(X)$.
On a $H=f(K)$.
\end{coro} 
\begin{proof}
Soit $x$ un point de~$X$.
Soit $Q$ un voisinage de $f(x)$ dans $P$ qui est un
polytope compact. Le domaine 
$W=f_\trop^{-1}(Q)$ de~$X$ est un voisinage de~$x$ et 
$\omega|_W=(f|_W)^{*}(\alpha|_Q)$.

Supposons que $x$ appartienne à~$K$.
Le compact~$K$ rencontre alors l'intérieur de~$W$,  
ce qui entraîne que $\omega|_W\neq 0$,
et donc que $\alpha|_Q\neq 0$. Il s'ensuit que $H$
rencontre~$Q$. Comme $Q$ peut être
choisi  arbitrairement petit, $f(x)$
appartient à~$H$,
d'où  l'inclusion $f(K)\subset H$. 

Réciproquement,
supposons que $f(x)\in H$.
Par définition de~$H$, 
la restriction à~$Q$ de la forme $\alpha$ est non nulle.
D'après la proposition~\ref{prop.nullite-formes},
la restriction de~$\omega$ à $f^{-1}(Q)$ est non nulle.
Par conséquent, $f^{-1}(Q)$ rencontre $K$, donc
$Q=f(f^{-1}(Q))$ rencontre $f(K)$.
Comme $Q$ est un voisinage arbitrairement petit
de~$f(x)$, celui-ci est adhérent à $f(K)$.
Puisque $f(K)$ est compact, $f(x)$ appartient à~$f(K)$.
\end{proof} 

%
%
%

\subsection{}\label{fonctions-loc-trop}
À toute carte G-tropicale $(f\colon V\to E,P)$ sur un domaine~$V$ 
de~$X$ et à toute fonction~$u$ sur~$P$,
associons la fonction $u\circ f$ sur~$V$. On définit
ainsi un  morphisme de préfaisceaux de $\mathscr F^{0,0}_\pf$
vers le G-faisceau des fonctions sur~$X$.
Il découle de la proposition~\ref{prop.nullite-formes}
que le morphisme de~$\mathscr F^{0,0}_{X_\groth}$ dans
le G-faisceau des fonctions sur~$X$ est injectif,
morphisme par lequel nous identifierons systématiquement 
une $(0,0)$-G-forme à une fonction.

Une section globale du faisceau~$\mathscr G^{0,0}_{X_\groth}$
est une fonction  sur~$X$ qui est G-localement
de la forme $\phi(f_1,\dots,f_n)$
où les~$f_i$ sont des fonctions paralinéaires sur~$X$,
et où $\phi$ appartient à la classe~$\mathscr G$. 
De telles fonctions seront appelées G-localement tropicales
de classe~$\mathscr G$.

%
 
\begin{defi}
Soit $\mathscr A$ la classe des fonctions lisses.
Les sections de 
$\mathscr A^{0,0}_{X_\groth}$
sont appelées \emph{fonctions G-lisses} 
sur~$X$.\index{fonction G-lisse sur un espace G-tropical}
\end{defi}
Les fonctions G-lisses sont donc les fonctions  sur~$X$ 
sont G-localement 
de la forme $\phi(f_1,\dots,f_n)$
où les $f_i$ sont des fonctions paralinéaires
et où $\phi$ est une fonction de classe~$\mathscr C^\infty$.

\subsection{}
Soit $\mathscr C$ la classe  des fonctions continues.
Toute fonction G-tropicale de classe~$\mathscr C$ est une
fonction continue sur l'espace topologique~$X$.

Il n'est pas vrai, en général, que toute fonction continue
sur~$X$ soit de cette forme.
(L'application $x\mapsto \abs{T(x)}$ sur~$\A^1$
montre qu'il existe des fonctions continues sur l'espace
topologique~$\A^1$ qui ne sont
pas constantes au voisinage d'un point rigide
et donc ne sont pas G-localement de classe~$\mathscr C$.)

Pour éviter toute ambiguïté, les fonctions de~$\mathscr C^{0,0}_{X_\groth}$
seront appelées \emph{G-localement tropicalement continues}.
\index{fonction!tropicalement continue}

On parlera aussi de G-formes 
à coefficients tropicalement continus
pour qualifier les G-formes
à coefficients dans~$\mathscr C$.
\index{G-forme!à coefficients tropicalement continus}
 
\subsection{}
Soit $\mathscr B$ la classe des fonctions boréliennes.
Les sections 
de~$\mathscr B^{0,0}_{X_\groth}$
seront appelées des fonctions \emph{G-localement tropicalement boréliennes}.
On parlera de G-formes à coefficients tropicalement boréliens
pour qualifier les G-formes
à coefficients dans~$\mathscr B$.
\index{G-forme!à coefficients tropicalement boréliens}
 
\begin{exem}\label{exem.indicatrice-borel}
Soit $X$ un espace G-tropical qui est associé, soit à un espace
paralinéaire, soit à un espace analytique.
\emph{La fonction indicatrice d'un domaine \emph{fermé}~$Y$ de~$X$
est G-localement tropicalement borélienne.}

Pour la démonstration, séparons les deux cas.

Traitons d'abord celui où $X$ est associé à un espace paralinéaire.
L'assertion étant G-locale, on peut supposer que $X$ est
un polyèdre compact de~$\R^n$. Dans ce cas, $Y$ est un polyèdre
compact contenu dans~$X$, défini par des inégalités entre
fonctions affines. Sa fonction caractéristique est donc borélienne.

Traitons maintenant le cas où $X$ est associé à un espace analytique.
Puisque l'assertion est G-locale, on peut supposer que $X$ est affinoïde.
D'après le théorème de Gerritzen--Grauert,
$Y$ est alors réunion finie
de domaines rationnels~$Y_i$ de~$X$;
comme l'ensemble des domaines rationnels est stable par intersection finie,
on peut même supposer que $Y$ est un domaine rationnel de~$X$.
Il existe donc des fonctions $f_1,\dots,f_m,g\in\mathscr O(X)$
telles que $(f_1,\dots,f_m,g)=(1)$
et des nombres réels $\lambda_1,\dots,\lambda_m>0$ 
tels que $Y$ soit défini par les inégalités 
$\abs{f_i}\leq \lambda_i \abs{g}$.

Soit $y\in X$. Si $y\not\in Y$, la fonction indicatrice de~$Y$
est nulle  au voisinage de~$y$, car $Y$ est fermé.
Si $y\in Y$, on a nécessairement $g(y)\neq 0$.  Quitte à
remplacer~$X$ par un voisinage affinoïde de~$y$ sur lequel 
$g$ est inversible, et à remplacer $f_i$ par $f_i/g$,
on suppose alors que $g=1$.
Si $\abs{f_i(y)}=0$, 
on peut encore restreindre~$X$ de sorte que l'inégalité
$\abs{f_i(y)}\leq \lambda_i$ sur tout~$X$ et retrancher cette inégalité
de la définition de~$Y$.
Nous sommes alors ramenés au cas où tous les~$f_i(y)$
sont non nuls et donc, quitte à restreindre encore~$X$,
que les fonctions~$f_i$ sont toutes inversibles sur~$X$.

Les fonctions $f_1,\dots,f_m$ définissent alors un moment
$f\colon X\to \gm^m$ qui tropicalise la fonction indicatrice~$\mathbf 1_Y$.
\end{exem}


\begin{prop}\label{lemm.otimes-coeff}
L'homomorphisme canonique
\[  \mathscr G_{X_{\mathrm G}}\otimes_{\mathscr A_{X_{\mathrm G}}} \mathscr A_{X_{\mathrm G}}^{p,q}\to\mathscr G^{p,q}_{X_{\mathrm G}} \]
est un isomorphisme.
\end{prop}
\begin{proof}
Commençons par une remarque. Soit $P$ un sous-espace paralinéaire 
d'un espace affine~$V$, et soit 
$x_1,\dots,x_m$ des coordonnées sur~$\langle P\rangle$.
Les formes 
\[ \di x_{i_1}\wedge\dots \wedge \di x_{i_p}
\wedge \dc x_{j_1}\wedge\dots\wedge \dc x_{j_q} \]
forment une famille génératrice 
de $\mathscr A^{p,q}_P(P)$ comme module sur
$\mathscr A_P(P)$, et une famille génératrice de $\mathscr G_P^{p,q}(P)$
comme module sur $\mathscr G_P(P)$;
si, de plus, $P$ est une cellule, ces formes constituent
des bases de ces modules.
Par suite, l'homomorphisme
\[ \mathscr G_P(P) \otimes_{\mathscr A_P(P)} \mathscr A^{p,q}_P(P)
 \to \mathscr G_P^{p,q}(P) \]
est surjectif, et c'est un isomorphisme si $P$ est une cellule. 

Démontrons ensuite que l'homomorphisme mentionné dans le lemme est surjectif.
Soit $V$ un domaine de~$X$, 
soit $\omega$ une section de $\mathscr G^{p,q}_{X_{\mathrm G}}(V)$;
il s'agit de prouver qu'elle appartient G-localement à l'image
de cet homomorphisme. On peut donc supposer que $V$ est compact
et G-tropicalise~$\omega$; la remarque précédente montre alors 
que $\omega$ appartient à l'image de $\mathscr G_{X_{\mathrm G}}(V)\otimes_{\mathscr A_{X_{\mathrm G}}(V)} \mathscr A^{p,q}_{X_{\mathrm G}}(V)$.

Démontrons enfin l'injectivité de cet homomorphisme.
Soit $V$ un domaine de~$X$,
soit $\eta$ une section sur~$V$ du produit tensoriel
dont l'image est nulle dans~$\mathscr G^{p,q}_{X_{\mathrm G}}(V)$;
démontrons que $\eta$ est nulle.
On peut raisonner G-localement
et donc supposer que $V$ est compact et qu'il existe
une carte G-tropicale $f\colon V\to \R^n$ tel que $P=f(V)$ soit une cellule,
un entier~$m\geq 0$ et pour $i\in\{1,\dots,m\}$,
des éléments $\phi_i\in\mathscr G_{P}(P)$ 
et $\alpha_i\in\mathscr A^{p,q}_P(P)$
tels que l'on ait $\eta=\sum_{i=1}^m f^*\phi_i \otimes f^*\alpha_i$.
Par hypothèse, $f^*(\sum_{i=1}^m \phi_i\alpha_i)=0$.
On a donc $\sum_{i=1}^m \phi_i\alpha_i=0$ 
(proposition~\ref{prop.nullite-formes}), 
d'où $\sum_{i=1}^m\phi_i\otimes\alpha_i=0$ en vertu de la remarque
faite au début de la preuve du lemme, d'où $\eta=0$.
\end{proof}

%
%

\section{Formes positives sur un espace G-tropical}

\begin{defi} Soit~$X$ un espace G-tropical
et soit~$\omega$ une G-forme symétrique sur~$X$
de type~$(p,p)$.
On dit que~$\omega$ est \emph{positive}
(\resp \emph{faiblement positive}, \resp \emph{fortement positive})
s'il existe 
un G-recouvrement~$(X_i)$ de~$X$ par des domaines et,
pour chaque indice~$i$, une carte G-tropicale $(f_i \colon X_i \ra E_i,
P_i)$ sur~$X_i$ et une forme~$\alpha_i$ de type~$(p,p)$ sur~$P_i$
qui est positive (\resp faiblement positive, \resp fortement
positive)  et telle que~$\omega|_{X_i}=f_i^*\alpha_i$. 
\end{defi} 

Dans la définition, il ne serait pas restrictif de supposer
que les~$X_i$ sont des domaines compacts 
et les $P_i$ égaux à $f_{i}(X_i)$.

\subsection{}\label{rema.image.inverse}
Soit $f\colon Y\to X$ un morphisme d'espaces G-tropicaux
et soit $\omega$ une forme de type~$(p,p)$ sur~$X$.
Si $\omega$ est positive 
(\resp faiblement positive, 
\resp fortement positive),
il en est de même de la forme~$f^*\omega$ sur~$Y$.

Cela résulte immédiatement de la définition.

\begin{lemm}\label{lemme-gpositif.compatibilite}
 Soit~$X$ un espace G-tropical,
soit~$(f\colon X\ra E,P)$
 une carte G-tropicale sur~$X$ 
 et soit~$\alpha$ une G-forme symétrique
 de type~$(p,p)$  sur~$P$.

 La G-forme~$f^*\alpha$ est positive
 (\resp faiblement positive, \resp fortement positive)
 si et seulement si c'est le cas de~$\alpha|_{f(V)}$
pour tout domaine compact~$V$ de~$X$.
\end{lemm} 
\begin{proof}
Supposons que pour tout domaine compact~$V$ de~$X$, la restriction de~$\alpha$ à~$f(V)$ soit positive. 
Soit alors $(X_i)$ un G-recouvrement de~$X$ par des domaines compacts; pour tout~$i$, posons $P_i=f_{i}(X_i)$, $f_i=f|_{X_i}$ et $\alpha_i=\alpha|_{P_i}$. 
Par hypothèse, $\alpha_i$ est une $(p,p)$-forme positive 
et l'on a $f^*\alpha|_{X_i}=f_i^*\alpha_i$.
La famille $(X_i)$,
les cartes G-tropicales $(f_i\colon X_i\to T,P_i)$ et les formes $\alpha_i$ 
attestent donc que $f^*\alpha$ est positive.

Réciproquement, supposons~$f^*\alpha$ positive
et soit $V$ un domaine compact de~$X$;
démontrons que la restriction de~$\alpha$ à~$f(V)$
est positive.
Par hypothèse, il existe un G-recouvrement~$(V_i)$
de~$V$ par des domaines compacts, 
et pour tout~$i$ une carte G-tropicale $f_i \colon V_i\ra E_i$
et une G-forme~$\alpha_i$ positive 
sur~$f_{i}(V_i)$ telle que~$f_i^*\alpha_i=(f^*\alpha)|_{V_i}$. 
Comme $V$ est compact, on peut  supposer que ce G-recouvrement est fini.

Pour tout~$i$, posons $h_i=(f_i,f|_{V_i})\colon V_i\to E_i\times E$,
notons $p_i$ et~$q_i$ les projections de $E_i\times E$ sur~$E_i$ et~$E$
respectivement; posons enfin $\beta_i=p_i^*\alpha_i$
et $\gamma_i=q_i^*\alpha$.
On a 
$
h_i^*\beta_i 
= f^*\alpha|_{V_i}
= h_i^*\gamma_i$.
Par suite, les restrictions à $h_{i}(V_i)$ des formes
$\beta_i$  et $\gamma_i$ sont égales (proposition~\ref{prop.nullite-formes}).
Comme~$V$ est égal à 
la réunion des~$V_i$, on a
\[f(V)=\bigcup f(V_i)=\bigcup q_{i}(h_{i}(V_i)).\]
Soit $Q$ l'espace paralinéaire somme disjointe
des $h_{i}(V_i)$, soit $\beta$
la forme sur~$Q$ donnée par~$\beta_i=\gamma_i$ sur $h_{i}(V_i)$.
Par construction, cette forme~$\beta$ est positive.
Soit $q\colon Q\to f(V)$ l'application paralinéaire
donnée par~$q_{i}$ sur~$h_{i}(V_i)$;
on a  $\beta=q^*\alpha$.
Comme $q$ est surjective et compacte, elle possède G-localement
des sections (prop.~\ref{prop.sections-G-locales}).
Par suite, $\alpha$ est positive.
 
Les cas des G-formes faiblement positives et des G-formes fortement positives
se traitent de même.
\end{proof}

\begin{exem}\label{exem.forme-positive}
Soit $X$ un espace G-tropical, soit $(u_1,\dots,u_n)$ une 
famille finie de fonctions G-localement tropicalement
de classe~$\mathscr C^1$ sur~$X$ 
et soit $M\colon X\to \MAT_n(\R)$ une fonction 
G-localement tropicale telle que $M(x)$ 
soit symétrique positive pour tout $x\in X$.
Alors, la G-forme 
\[ \sum_{i,j} M_{i,j} \di u_i \wedge \dc u_j \]
sur~$X$ est positive.

L'assertion à prouver est G-locale sur~$X$.
On peut donc supposer que $X$ est compact
et que $u_1,\dots,u_n,M$ sont G-tropicalisées
par une carte G-tropicale commune $f\colon X\to \R^p$ tel que $f(X)$
soit une cellule~$C$ de dimension~$p$ sur~$\R^p$.
Il existe alors des fonctions $v_1,\dots,v_n,N$ sur~$C$,
les~$v_i$ étant de classe~$\mathscr C^1$, telles que
$u_i=f^*v_i$ pour $i\in\{1,\dots,n\}$ et $M=f^*N$.
En particulier, $N(y)$ est symétrique positive pour tout $y\in C$.

D'après le lemme~\ref{lemme.image.inverse}
appliqué à la forme positive $\sum N_{i,j}(y) \di y_i\wedge\dc y_j$
et à l'application linéaire de~$\R^p$ dans~$\R^n$
de matrice $(\partial v_i/\partial y_j)$,
la forme de type~$(1,1)$ sur $\R^p$ donnée par 
$\sum N_{i,j} (y) \di v_i(y)\wedge \dc v_j(y)$
est positive.
Par suite, la forme de type~$(1,1)$
$\sum N_{i,j}  \di v_i\wedge \dc v_j$
sur~$C$ est positive.
L'assertion voulue résulte alors de la formule
\[\sum_{i,j} M_{i,j} \di u_i \wedge \dc u_j 
=f ^* (\sum_{i,j} N_{i,j}  \di v_i \wedge \dc v_j )\]
et de la définition d'une forme positive.
\end{exem}

\begin{prop}\label{prop.produit-formes-positives-berk}
Soit $X$ un espace G-tropical, 
soit $\omega_1,\dots,\omega_r$ des G-formes  sur~$X$,
et soit $\omega=\omega_1\wedge\dots\wedge\omega_r$.

\begin{enumerate}
\item
Si les $\omega_j$
sont fortement positives, alors $\omega$
est fortement positive.

\item
Si les $\omega_j$
sont positives, alors $\omega$ est positive.

\item
Si les $\omega_j$ sont faiblement positives
et si au plus une d'entre elles n'est pas fortement positive, 
alors $\omega$ est faiblement positive.
\end{enumerate}
\end{prop}
\begin{proof}
Cela découle de la proposition~\ref{prop.produit-formes-positives-pl}.
\end{proof}

\section{Topologies sur l'espace des G-formes}

Dans ce paragraphe, on définit des topologies~$\mathscr C^m$
sur les espaces de G-formes G-lisses. 
Mutatis mutandis, ces constructions fournissent
des topologies~$\mathscr C^m$ sur les espaces de G-formes
à coefficients formés de fonctions~$\mathscr C^m$;
les modifications correspondantes sont laissées au lecteur.

\subsection{}
Soit $C$ une cellule compacte de~$\R^c$, de dimension~$c$.

Pour tout entier~$m\geq 0$,
on dispose sur l'ensemble $\mathscr C^\infty(C)$
des fonctions de classe~$\mathscr C^\infty$ sur~$C$
de la semi-norme $\norm{\cdot}_{C,m}$ donnée par la borne supérieure des
valeurs des dérivées partielles d'ordre~$\leq m$;
c'est une norme.

Soit $(p,q)$ un couple d'entiers naturels.
L'ensemble des formes lisses de type~$(p,q)$ sur~$C$ 
est un module libre sur $\mathscr C^\infty(C)$,
de base la famille $(\di x_I\wedge \dc x_J)$,
où $(I,J)$ parcourt l'ensemble des couples de multi-indices
de longueurs~$p$ et~$q$ respectivement.
On définit alors la semi-norme~$\norm{\cdot}_{C,m}$
sur l'espace $\mathscr A^{p,q}(C)$ des formes lisses sur~$C$,
de type~$(p,q)$, par la borne supérieure des semi-normes des coefficients
d'une forme dans cette base; c'est encore une norme.

La topologie correspondante est appelée topologie~$\mathscr C^m$
sur~$\mathscr A^{p,q}(C)$.  Elle est normable et en particulier séparée.

Si $m\leq n$,
la topologie~$\mathscr C^m$ est moins fine que la topologie~$\mathscr C^n$.
La topologie qu'elles engendrent, définie par toutes les
normes $\norm{\cdot}_{C,m}$ pour $m\in\N$,
est appelée topologie~$\mathscr C^\infty$; elle est métrisable
et en particulier séparée.

\subsection{}
Soit $D$ une cellule compacte de~$\R^d$, de dimension~$d$
et soit $f\colon C\to D$ une application affine.
Pour tout entier~$m$, il existe un nombre réel~$M_{m}$
tel que 
$\norm{f^*\omega}_{C,m}\leq M_{m} \norm{\omega}_{D,m}$
pour toute forme $\omega\in\mathscr A^{p,q}(D)$.

Si $f$ est surjective, elle possède une section affine~$g$;
on en déduit l'existence, pour tout entier~$m$,
d'un nombre réel~$M'_m$ 
tel que 
$ \norm{\omega}_{D,m} \leq M'_m \norm{f^*\omega}_{C,m}$.
Autrement dit, l'application~$f^*$ est injective est stricte.

En particulier, les automorphismes affines d'une cellule~$C$
induisent des homéomorphismes pour les topologies~$\mathscr C^m$.

\subsection{}\label{subsec.topologie-forme-pl}
Soit $X$ un espace paralinéaire.
Pour tout élément~$m$ de~$\N\cup\{+\infty\}$,
on définit la topologie $\mathscr T_{m}$ sur $\mathscr A_\groth^{p,q}(X)$
comme la topologie la moins fine rendant continues
les applications $f^*$ où $f\colon C\to X$ est une application
paralinéaire d'une cellule compacte~$C$ de dimension~$d$
de~$\R^d$ dans~$X$.

À toute application $f\colon C\to X$ et tout entier~$r$, on peut associer
une (classe de) semi-norme $\mathscr C^r$,
$\norm\cdot_{f,m}$ sur~$\mathscr A^{p,q}_\groth(X)$,
image directe de la norme~$\mathscr C^r$ sur $\mathscr A^{p,q}_\groth(C)$.
La topologie~$\mathscr T_m$ est ainsi définie par ces seminormes
$\norm{\cdot}_{f,r}$, où $r$ parcourt les entiers~$\leq m$.

Pour définir cette topologie,
on peut en fait se limiter à une famille d'applications~$f$ 
qui sont des immersions
et dont les images constituent un G-recouvrement de~$X$.

Si $X$ est compact et $m$ est fini, 
on en déduit que cette topologie est normable.
Dans ce cas, on appellera norme~$\mathscr C^m$ standard sur~$X$
toute norme sur $\mathscr A_\groth^{p,q}(X)$.
Un exemple typique de norme standard consiste à prendre la borne
supérieure des normes $\norm{\cdot}_{f,m}$ où $f\colon C\to X$
parcourt un décomposition cellulaire de~$X$.

Cette topologie est métrisable si $X$ est paracompact.

Cette topologie est séparée.

\subsection{}
Soit $f\colon P\to Q$ une application paralinéaire entre espaces
paralinéaires.

Pour tout~$m$, l'application $\omega\mapsto f^*\omega$
est continue pour les topologies~$\mathscr T_{P,m}$
et~$\mathscr T_{Q,m}$.

Si $f$ est surjective, elle admet G-localement des sections paralinéaires
et l'application continue $f^*$ est injective et stricte.

\subsection{}
Soit $X$ un espace G-tropical.
Soit $m$ un élément de~$\N\cup\{+\infty\}$.
Soit $V$ un domaine de~$X$, 
soit $f\colon V\to E$ une application paralinéaire
à valeurs dans un espace affine~$E$.

Soit $\mathscr A^{p,q}_{f,\groth}(X)$
le sous-espace de $\mathscr A^{p,q}_\groth(X)$ constitué des G-formes
$\omega$ qui sont tropicalisées par~$f$ sur tout domaine compact de~$V$,
c'est-à-dire telles que pour tout domaine compact~$W$ de~$V$,
il existe $\alpha_W\in\mathscr A^{p,q}(f(W))$ vérifiant
$\omega|_W=f^*\alpha_W$. Rappelons qu'alors $\alpha_W$ est déterminé
par~$\omega|_W$ (proposition~\ref{prop.nullite-formes}).

Si $P$ est une partie paralinéaire de~$E$ contenant~$f(V)$
et $\alpha \in\mathscr A^{p,q}(P)$ est une G-forme,
alors $f^*\alpha\in\mathscr A^{p,q}_{f,\groth}(V)$
et $(f^*\alpha)_W = \alpha|_{f(W)}$ pour tout domaine compact~$W$ de~$V$.
(En général, une forme $\omega\in \mathscr A^{p,q}_{f,\groth}(X)$ 
n'est pas de la forme $f^*\alpha$, pour $\alpha$ comme ci-dessus,
même si $V=X$ et $f(V)=P$.
On prend par exemple pour~$X$ l'espace~$\R$ muni de la topologie
discrète et de sa structure paralinéaire évidente,
$f\colon X\to \R$ l'application identique, qui est paralinéaire;
les domaines compacts de~$X$ sont les ensembles finis;
toute fonction sur~$X$ est G-lisse et est $f$-tropicale 
sur tout domaine compact; bien peu sont de la forme $f^*\phi$
où $\phi$ est une fonction lisse sur~$\R$.)

Soit $W$ un domaine compact de~$V$. On considère le diagramme
\[ \mathscr A^{p,q}_{f,\groth}(X) \to \mathscr A^{p,q}_{f,\groth}(W)
\leftarrow \mathscr A^{p,q}(f(W)) \]
défini par l'application de restriction et l'application~$f^*$.
L'application~$f^*\colon\mathscr A^{p,q}(f(W))\to\mathscr A^{p,q}_{f,\groth}(W)$
est bijective; on munit son but de la topologie déduite 
par transport de structure de
la topologie~$\mathscr C^m$ sur $\mathscr A^{p,q}(f(W))$.
On munit alors $\mathscr A^{p,q}_{f,\groth}(X)$
de la topologie la moins fine rendant continue la restriction à~$W$
pour tout domaine compact~$W$ de~$V$;
on la note $\mathscr T_{f,m}$.

Pour cette topologie, l'adhérence de la G-forme nulle
est l'ensemble des G-formes~$\omega\in\mathscr A^{p,q}_\groth(X)$
telles que $\omega|_V=0$.
En effet, la topologie de $\mathscr A^{p,q}_{f,\groth}(W)$
est \emph{normable}, donc séparée;
par définition de la topologie de~$\mathscr A^{p,q}_{f,\groth}(W)$,
l'adhérence de~$0$ est l'ensemble des formes
dont toutes les restrictions aux domaines compacts~$W$ de~$V$ sont nulles;
c'est donc l'ensemble des éléments de~$\mathscr A^{p,q}_{f,\groth}(X)$
dont la restriction à~$V$ est nulle.
D'autre part, toute forme de $\mathscr A^{p,q}_\groth(X)$
dont la restriction à~$V$ est nulle appartient à~$\mathscr A^{p,q}_{f,\groth}(X)$ puisqu'elle est tropicalisée par la forme nulle sur $f(W)$,
pour tout domaine compact~$W$ de~$V$.

\subsection{}
Considérons deux applications paralinéaires $f\colon U\to E$ et 
$g\colon V\to F$ définies sur des domaines~$U$ et~$V$ de~$X$.
Supposons que $V$ soit contenu dans~$U$ et qu'il existe
une application affine 
$q\colon F\to E$ une application affine telle que $q\circ g=f|_V$.
Une forme G-tropicalisée par~$f$ sur tout domaine
compact de~$U$ est également G-tropicalisée
par~$g$ sur tout domaine compact de~$V$, 
et l'inclusion $\mathscr A^{p,q}_{\groth, f}(X)
\hookrightarrow \mathscr A^{p,q}_{\groth,g}(X)$ est continue.

En effet, il suffit de s'assurer que pour tout domaine
compact~$W$ de~$V$, l'application composée
$\mathscr A^{p,q}_{f,\groth}(X)\hookrightarrow\mathscr A^{p,q}_{g,\groth}(X)
\xrightarrow{\text{restr}}\mathscr A^{p,q}_{g,\groth}(W)$ est continue.
Cette application est l'application composée
$\mathscr A^{p,q}_{f,\groth}(X) \to \mathscr A^{p,q}_{f,\groth}(W)
 \hookrightarrow \mathscr A^{p,q}_{g,\groth}(W)$; 
la première flèche étant continue, il suffit alors de vérifier que
l'application d'inclusion
$\mathscr A^{p,q}_{f,\groth}(W)\hookrightarrow \mathscr A^{p,q}_{g,\groth}(W)$
est continue.
Par définition de la topologie quotient sur $\mathscr A^{p,q}_{f,\groth}(W)$,
il suffit de vérifier que l'application 
composée
$ \mathscr A^{p,q}(f(W)) \xrightarrow{f^*} \mathscr A^{p,q}_{f,\groth}(W)\hookrightarrow \mathscr A^{p,q}_{g,\groth}(W)$
est continue.
Or cette application est également l'application composée
$ \mathscr A^{p,q}(f(W)) \xrightarrow{q^*}  \mathscr A^{p,q}(g(W))
 \xrightarrow{g^*}  \mathscr A^{p,q}_{g,\groth} (W)$.
Ces deux applications~$q^*$ et~$g^*$ étant continues, l'assertion en résulte.

\subsection{}
Par abus de langage, 
appelons \emph{jeu de cartes} G-tropicales\index{cartes!jeu de}
sur~$X$ toute famille d'applications paralinéaires
définies sur des domaines de~$X$ à valeurs dans des espaces affines.
Nous ne spécifions pas de partie paralinéaire au but, 
si bien que les objets d'un jeu de cartes ne sont pas des cartes…
En pratique, nous ne considérerons que les restrictions
des fonctions d'un jeu de cartes à des domaines compacts;
leur image est alors une partie paralinéaire, 
ce qui permet de les considérer comme des cartes.

Un jeu de cartes est dit \emph{G-couvrant} si ses domaines forment
un G-recouvrement de~$X$.

\subsection{}\label{ss.topologie-jeu-de-carte}
Soit $\mathscr D$ un jeu de cartes sur~$X$.
Soit $\mathscr A^{p,q}_{\groth,\mathscr D}(X)$
l'intersection, dans $\mathscr A^{p,q}_\groth(X)$,
des sous-espaces $\mathscr A^{p,q}_{f,\groth}(X)$,
lorsque $f$ parcourt l'ensemble des cartes de~$\mathscr D$.

On le munit de la topologie la moins fine
rendant continues  les applications d'inclusion
$\mathscr A^{p,q}_{\mathscr D,\groth}(X)\to \mathscr A^{p,q}_{f,\groth}(X)$
est continue, lorsque ce dernier espace est muni de la
topologie~$\mathscr T_{f,m}$,  
où $f$ est une carte de~$\mathscr D$.
Cela fournit ainsi une topologie~$\mathscr T_{\mathscr D,m}$
sur $\mathscr A^{p,q}_{\groth,\mathscr D}(X)$
qu'on appelle la topologie~$\mathscr C^m$.

\subsection{}
L'adhérence de la forme nulle dans~$\mathscr A^{p,q}_{\groth,\mathscr D}(X)$
est l'ensemble des G-formes~$\omega$  qui appartiennent à
l'adhérence de~$0$ pour toutes les topologies~$\mathscr T_{f,m}$,
où $f$ est un membre du jeu de carte~$\mathscr D$.
Compte tenu des définitions et du calcul de l'adhérence de~$0$
dans~$\mathscr A^{p,q}_{f,\groth}(X)$,
ce sont les G-formes~$\omega$ dont la restriction à tout domaine~$V$
d'un membre de~$\mathscr D$ est nulle.

En particulier, si le jeu de cartes~$\mathscr D$ est G-couvrant,
la topologie~$\mathscr C^m$ sur $\mathscr A^{p,q}_{\mathscr D,\groth}(X)$
est séparée.

\begin{prop}\label{prop.continuite-wedge-jeu}
Soit $\mathscr D$ un jeu de cartes sur un espace G-tropical~$X$
et soit $n\in\N\cup\{+\infty\}$.
Le produit extérieur des G-formes induit 
des applications bilinéaires continues 
\[ \mathscr A^{p,q}_{\groth,\mathscr D} (X) \times \mathscr A^{p',q'}_{\groth,\mathscr D} (X)
 \to \mathscr A^{p+p',q+q'}_{\groth,\mathscr D}(X). \]
\end{prop}
\begin{proof}
Soit $f\colon U\to E$ une carte de~$\mathscr D$,
soit $m$ un entier naturel tel que $m\leq n$,
et soit $K$ un domaine compact de~$U$.
Pour tout couple $(p,q)$ d'entiers naturels,
on déduit d'un décomposition cellulaire convenable de l'espace
paralinéaire~$f(K)$ une norme sur $\mathscr A^{p,q}(f(K))$,
et on en tire une semi-norme $\norm\cdot_{K,m}$
sur $\mathscr A^{p,q}_{\groth,f}(X)$,
et en particulier sur $\mathscr A^{p,q}_{\groth,\mathscr D}(X)$.
Au niveau paralinéaire, le produit des formes définit une application
bilinéaire
\[ \mathscr A^{p,q}(f(K)) \times \mathscr A^{p',q'} (f(K)) \to 
\mathscr A^{p+p',q+q'}(f(K)) \]
qui est continue pour la topologie~$\mathscr C^m$.
On en déduit  que l'application bilinéaire de l'énoncé est continue,
lorsque ces espaces sont munis des semi-normes $\norm{\cdot}_{K,m}$
correspondantes. Le résultant s'en déduit formellement.
\end{proof}

\subsection{}
Soit $u\colon X\to Y$ un morphisme d'espaces G-tropicaux.

Soit $f\colon U\to E$ et $g\colon V\to F$ 
des applications paralinéaires
définies sur des domaines~$U$ de~$X$ et~$V$ de~$Y$ respectivement;
on dit que $f$ et~$g$ sont compatibles relativement à~$u$
si $u(U)\subset V$ et s'il existe 
une application affine $q\colon E\to F$
telle que $q\circ f=g\circ u|_U$.

Soit $\mathscr D$ un jeu de cartes G-tropicales sur~$X$ et 
soit $\mathscr E$ un jeu de cartes G-tropicales sur~$Y$;
on dit que $\mathscr D$ et~$\mathscr E$ 
sont \emph{compatibles} relativement à~$u$
si pour tout élément~$f$ de~$\mathscr D$,
il existe un élément~$g$ de~$\mathscr E$ tel
que $f$ et~$g$ soient compatibles relativement à~$u$.
On notera cette condition $\mathscr D\preceq_u \mathscr E$, 
voire $\mathscr D\preceq \mathscr E$
s'il n'y a pas d'ambiguïté sur~$u$ (en particulier lorsque $u=\id$).
Cela entraîne que  pour toute G-forme 
$\omega\in \mathscr A^{p,q}_{\mathscr E,\groth}(Y)$, la G-forme $u^*\omega$
sur~$X$ appartient à $\mathscr A^{p,q}_{\mathscr D,\groth}(X)$.

On dit que $\mathscr D$ \emph{couvre tropicalement}~$\mathscr E$ 
relativement à~$u$ si pour toute application paralinéaire
$g\colon V\to F$ appartenant à~$\mathscr E$
et tout domaine compact~$K$ de~$V$,
l'ensemble~$g(K)$ est G-recouvert par des parties
paralinéaires compactes de la forme $g\circ u(W)$,
où $W$ est un domaine compact du domaine d'un élément $f\colon U\to E$
de~$\mathscr D$ qui est compatible à~$g$ relativement à~$u$
tel que $u(W)\subset K$.
Par compacité, cela revient à dire que $g(K)$
est la réunion d'une famille finie  de tels domaines.

Cette propriété est en particulier vérifiée
si $\mathscr D$ \emph{couvre}~$\mathscr E$ relativement à~$u$,
c'est-à-dire 
si pour toute application paralinéaire
$g\colon V\to F$ appartenant à~$\mathscr E$,
tout domaine compact~$K$ de~$V$ est G-recouvert
par des parties compactes de la forme $u(W)$,
où $W$ est un domaine compact du domaine d'un élément $f\colon U\to E$
de~$\mathscr D$ qui est compatible à~$g$ relativement à~$u$
tel que $u(W)\subset K$.



\begin{prop}\label{prop.raffine-topologie}
Soit $u\colon X\to Y$ un morphisme d'espaces G-tropicaux,
soit $\mathscr D$ et $\mathscr E$ des jeux de cartes sur~$X$ et~$Y$ 
tels que $\mathscr D\preceq_u \mathscr E$.
Soit $m \in\N\cup\{+\infty\}$.

\begin{enumerate}
\item
L'application $u^*\colon \mathscr A^{p,q}_{\mathscr E,\groth}(Y)
\to \mathscr A^{p,q}_{\mathscr D,\groth}(X)$ est continue
pour les topologies~$\mathscr C^m$ sur ces espaces.

%
\item
Si, de plus, $\mathscr D$ couvre tropicalement~$\mathscr E$ 
relativement à~$u$, cette application linéaire vérifie
la propriété topologique suivante: pour tout voisinage~$\Omega$ de~$0$
dans~$\mathscr A^{p,q}_{\groth,\mathscr E}(Y)$,
il existe un voisinage~$\Upsilon$ de~$0$ 
dans $\mathscr A^{p,q}_{\groth,\mathscr D}(X)$ tel 
que toute G-forme $\omega\in\mathscr A^{p,q}_{\groth,\mathscr E}(Y)$ 
pour laquelle $u^*\omega\in\Upsilon$ vérifie $\omega\in\Omega$.
En particulier, l'application linéaire continue~$u^*$ est stricte
et l'image réciproque de l'adhérence de~$0$ par~$u^*$ est l'adhérence de~$0$.
\end{enumerate}
\end{prop}
\begin{proof}
\begin{enumerate}
\item
La topologie~$\mathscr C^m$ de~$\mathscr A^{p,q}_{\groth,\mathscr D}(X)$
est la topologie initiale des topologies~$\mathscr C^m$
sur $\mathscr A^{p,q}_{\groth,f}(X)$, pour $f\in\mathscr D$.
Pour démontrer que $u^*$ est continue, il suffit de démontrer
que pour tout $f\in\mathscr D$,
l'applications $u^*\colon\mathscr A^{p,q}_{\groth,\mathscr E}(Y)
\to \mathscr A^{p,q}_{\groth,f}(X)$ est continue.

Soit $f\colon U\to E$ une telle carte.
Soit $W$ un domaine compact de~$U$. Il suffit de démontrer
que l'application $u|_W^*\colon\mathscr A^{p,q}_{\groth,\mathscr E}(Y)
\to \mathscr A^{p,q}_{\groth,f}(W)$ est continue;
rappelons que la topologie de 
ce dernier espace $ \mathscr A^{p,q}_{\groth,f}(W)$ 
est définie par transport de structure depuis celle
de~$\mathscr A^{p,q}(f(W))$
par l'application $\alpha\mapsto f^*\alpha$.

Comme $\mathscr D$ et~$\mathscr E$ sont compatibles relativement à~$u$,
il existe une carte $g\colon V\to F$ de~$\mathscr E$
et une application affine $q\colon E\to F$ 
telle que $q\circ f=g\circ u|_U$.

Soit $K$ un domaine compact de~$V$ qui contient~$u(W)$.
On a $q(f(W))=g(u(W))\subset g(K)$, donc 
l'application $q^*\colon \mathscr A^{p,q}(g(K))\to\mathscr A^{p,q}(f(W))$
est continue (\S\ref{subsec.topologie-forme-pl}).
L'application  composée
\[ \mathscr A^{p,q}_{\groth,\mathscr E}(Y)
\to \mathscr A^{p,q}_{\groth,g}(V)
\to \mathscr A^{p,q}_{\groth,g}(K)
\simeq \mathscr A^{p,q}(g(K)) 
\xrightarrow{q^*} \mathscr A^{p,q}(f(W)) \simeq \mathscr A^{p,q}_{\groth,f}(W)
\]
est donc continue, et ce n'est autre que~$u^*$.

\item
Faisons maintenant l'hypothèse supplémentaire que $\mathscr D$ 
couvre tropicalement~$\mathscr E$ relativement à~$u$.

La propriété demandée étant compatible aux intersections finies,
on peut supposer qu'il existe 
une carte $g\colon V\to F$ de~$\mathscr E$,
un domaine compact~$K$ de~$V$, un entier naturel~$n\leq m$
et un nombre réel~$\eps>0$,
tels que $\Omega$ soit l'ensemble
des G-formes $\omega\in\mathscr A^{p,q}_{\groth,\mathscr E}(Y)$
telles que $\norm{\omega}_{g(K),n}<\eps$,
où $\norm{\cdot}_{g(K),n}$ est une norme sur 
l'espace $\mathscr A^{p,q}(g(K))$
associée à une décomposition cellulaire convenable de~$g(K)$,
transportée en une semi-norme sur 
$\mathscr A^{p,q}_{\groth,\mathscr E}(Y)$ par le diagramme
\[ \mathscr A^{p,q}_{\groth,\mathscr E}(Y)
\to \mathscr A^{p,q}_{\groth,g}(V)
\to \mathscr A^{p,q}_{\groth,g}(K)
\simeq \mathscr A^{p,q}(g(K)) .
\]

Posons $Q=g(K)$.
Par hypothèse, $\mathscr D$ couvre tropicalement~$\mathscr E$
relativement à~$u$, de sorte que $Q$ est G-recouvert 
par des parties paralinéaires
compactes de la forme $g(u(W))$, où $W$ est un domaine compact
du domaine d'un élément $f\colon U \to E$ de~$\mathscr D$
qui est compatible à~$g$ relativement à~$u$ et tel que $u(W)\subset K$.
Comme $g(K)$ est une partie paralinéaire compacte,
il existe ainsi une famille finie
$(f_i\colon U_i\to E_i)$ de telles cartes,
et, pour tout~$i$, un domaine compact~$W_i$ de~$U_i$
tel que $u(W_i)\subset K$,
tels que $g(K)$ soit la réunion des $g(u(W_i))$.
Pour tout~$i$, soit $q_i\colon E_i\to F$ une application affine
telle que $q_i\circ f_i=g\circ u|_{U_i}$;
posons aussi $P_i=f(W_i)$.

Par hypothèse, la réunion des~$q_i(P_i)$ est égale à~$Q$.
On choisit pour chaque~$P_i$ une norme $\norm\cdot_{P_i,n}$
qui induit la topologie~$\mathscr C^n$ sur $\mathscr A^{p,q}(P_i)$.

Soit $P$ l'espace paralinéaire compact somme disjointe des~$P_i$
et soit $q\colon P\to Q$ l'application paralinéaire déduite des~$q_i$.
L'espace $\mathscr A^{p,q}(P)$ 
est la somme directe des $\mathscr A^{p,q}(P_i)$;
on le munit de la norme~$\norm\cdot_{P,n}$, sup des normes~$\norm{\cdot}_{P_i,n}$.
Comme l'application~$q$ 
admet G-localement des sections, 
il existe un nombre réel~$c>0$ tel que
$\norm{q^*\alpha}_{P,n}\geq c \norm{\alpha}_{Q,n}$
pour toute G-forme $\alpha\in\mathscr A^{p,q}(Q)$.

Les normes $\norm{\cdot}_{P_i,n}$ sur~$\mathscr A^{p,q}(P_i)$
induisent des semi-normes sur $\mathscr A^{p,q}_{\groth,\mathscr D}(X)$
que l'on note de la même façon.
Soit $\Upsilon$ l'ensemble des G-formes~$\eta$ 
dans $\mathscr A^{p,q}_{\groth,\mathscr D}(X)$
telles que $\norm{\eta}_{P_i,n}<\eps/c$ pour tout~$i$.
Soit $\omega\in\mathscr A^{p,q}_{\groth,\mathscr E}(Y)$
telle que $u^*\omega\in\Upsilon$; démontrons que $\omega\in\Omega$.

Il existe une unique G-forme $\alpha\in\mathscr A^{p,q}(Q)$
telle que $\omega|_V=g^*\alpha$, et $\norm{\omega}_{Q,n}(\omega)
=\norm\omega_{Q,n}$ par construction.
Pour tout~$i$, on a $u(W_i)\subset K$, donc 
\[  u^*\omega|_{W_i} = u^* (g^*\alpha)|_{W_i} = (g\circ u)^*\alpha|_{W_i}
 = (q_i\circ f_i)^*\alpha |_{W_i}=f_i^* (q_i^*\alpha|_{P_i}). \]
Par définition de~$\Upsilon$, on a 
$ \norm{u^*\omega}_{P_i,n}<\eps/c$,
donc $\norm{q_i^*\alpha}_{P_i,n}<\eps/c$.
Ainsi, $\norm{q^*\alpha}_{P,n}<\eps/c$, si bien que $\norm\alpha_{Q,n}<\eps$.
Ainsi, $\norm\omega_{Q,n}<\eps$ et $\omega\in\Omega$.

Ceci établit la propriété topologique de l'énoncé.
Les conséquences données en résultent directement.\qedhere
\end{enumerate}
\end{proof}

\begin{coro}
\begin{enumerate}
\item
Soit $\mathscr D$ et $\mathscr E$ des jeux de cartes sur~$X$,
et soit $n\in\N\cup\{+\infty\}$.
Si $\mathscr D\preceq \mathscr E$ et si $\mathscr D$ est G-couvrant,
alors l'inclusion $\mathscr A^{p,q}_{\groth,\mathscr E}\subset
 \mathscr A^{p,q}_{\groth,\mathscr D}$ est continue et stricte.

\item
Soit $(\mathscr D_i)$ une famille de jeux de cartes G-couvrants sur~$X$
et soit $n\in\N\cup\{+\infty\}$.
Les topologies $\mathscr T_{\mathscr D_i,n}$ coïncident
sur l'intersection $\bigcap_i \mathscr A^{p,q}_{\groth,\mathscr D_i}$.
\end{enumerate}
\end{coro}
\begin{proof}
\begin{enumerate}
\item
Comme le jeu~$\mathscr D$ est G-couvrant, il couvre le jeu~$\mathscr E$.
L'assertion découle ainsi de la proposition~\ref{prop.raffine-topologie}.

\item
Soit $\mathscr E$ la réunion des familles $\mathscr D_i$.
Pour tout~$i$, le jeu~$\mathscr E$ raffine le jeu~$\mathscr D_i$,
et l'on a $\bigcap_i \mathscr A^{p,q}_{\groth,\mathscr D_i}
=\mathscr A^{p,q}_{\groth,\mathscr E}$.
L'assertion résulte ainsi de la précédente.\qedhere
\end{enumerate}
\end{proof}

\begin{prop}\label{prop.continuite-di-dc-jeu}
Soit $X$ un espace G-tropical, soit $m\in\N\cup\{+\infty\}$.
Soit $\mathscr D$ un jeu de cartes sur~$X$.

\begin{enumerate}
\item Pour toute G-forme $\omega\in\mathscr A^{p,q}_{\mathscr D,\groth}(X)$,
la forme $\mathrm J\omega$ appartient à~$\mathscr A^{q,p}_{\mathscr D,\groth}(X)$. L'application~$\mathrm J$ ainsi définie est linéaire et continue.
\item Supposons $m\geq 1$.
Pour toute G-forme $\omega\in\mathscr A^{p,q}_{\mathscr D,\groth}(X)$,
on a $\di\omega \in\mathscr A^{p+1,q}_{\mathscr D,\groth}(X)$
et $\dc\omega \in\mathscr A^{p,q+1}_{\mathscr D,\groth}(X)$.
Les applications~$\di$ et~$\dc$ ainsi définies sont continues
lorsqu'on munit $\mathscr A^{p,q}_{\mathscr D,\groth}(X)$
de la topologie~$\mathscr C^m$ et les
espaces $\mathscr A^{p+1,q}_{\mathscr D,\groth}(X)$
et $\mathscr A^{p,q+1}_{\mathscr D,\groth}(X)$
de la topologie~$\mathscr C^{m-1}$.
\end{enumerate}
\end{prop}
\begin{proof}
La démonstration découle de la définition des semi-normes.
\end{proof}

\section{Passages à la limite}

\subsection{}
Soit $m\in\N\cup\{+\infty\}$.
L'espace $\mathscr A^{p,q}_\groth(X)$ est la réunion
de ses sous-espaces $\mathscr A^{p,q}_{\groth,\mathscr D}(X)$,
où $\mathscr D$ parcourt l'ensemble des jeux de cartes G-tropicales
qui sont {G-couvrants} et
on munit $\mathscr A^{p,q}_\groth(X)$ de la topologie
localement convexe la plus fine pour laquelle toutes les inclusions
correspondantes
$\mathscr A^{p,q}_{\groth,\mathscr D}\subset \mathscr A^{p,q}_\groth(X)$
sont continues, lorsque $\mathscr A^{p,q}_{\groth,\mathscr D}$
est muni de la topologie~$\mathscr T_{\mathscr D,m}$.
On dira que c'est la topologie~$\mathscr C^m $ 
sur~$\mathscr A^{p,q}_\groth(X)$.

Si $m\leq n$, la topologie~$\mathscr C^n$ sur~$\mathscr A^{p,q}_\groth(X)$
est plus fine que sa topologie~$\mathscr C^m$.

Pour qu'une application linéaire $T\colon\mathscr A^{p,q}_{\groth}(X)\to E$
soit continue pour la topologie~$\mathscr T_{\mathscr D,m}$, 
$E$ étant un espace vectoriel localement convexe,
il faut et il suffit que sa restriction à $\mathscr A^{p,q}_{\groth,\mathscr D}$ soient continues, pour tout jeu de cartes G-tropicales~$\mathscr D$ sur~$X$
qui est G-couvrant.

\subsection{}
Pour notre théorie des courants, nous aurons besoin de généraliser cette construction dans deux directions différentes:
\begin{itemize}
\item Nous restreindre à des G-formes tropicalisées par des jeux de cartes 
particuliers;
\item Ne considérer que des G-formes à support compact.
\end{itemize}
Pour traiter cela de façon générale et uniforme, nous introduisons
les familles de jeux de cartes et les familles de supports le long
desquels nous formons les colimites convenables.

\subsection{}
Rappelons qu'on appelle
\emph{famille de supports}\index{supports!famille de ---} dans~$X$
un ensemble~$\Phi$ de parties fermées de~$X$
vérifiant les deux propriétés suivantes:
\begin{enumerate}\def\theenumi{\roman{enumi}}
\item
Toute réunion finie d'éléments de~$\Phi$ appartient à~$\Phi$;
\item
Toute partie fermée d'un élément de~$\Phi$ appartient à~$\Phi$.
\end{enumerate}

Les familles suivantes sont les principales familles de supports
que nous considérerons:
\begin{enumerate}
\item L'ensemble de toutes les parties fermées de~$X$; 
\item Si $T$ est une partie de~$X$, 
l'ensemble des parties fermées de~$X$ qui sont contenues dans~$T$;
\item L'ensemble des parties compactes de~$X$ (rappelons que $X$
est supposé topologiquement séparé).
\end{enumerate}
La première et la  dernière de ces familles admettent 
la généralisation commune suivante:
\begin{enumerate}
\setcounter{enumi}{3}
\item Soit $f\colon X\to Y$ un morphisme d'espaces G-tropicaux.
L'ensemble des parties fermées de~$X$ qui sont topologiquement propres sur~$Y$
est une famille de supports dans~$X$.
\end{enumerate}
D'autres familles seraient possibles, par exemple, l'ensemble
des parties fermées de~$X$ qui sont dénombrables à l'infini.

\subsection{}
Soit $\Gamma$ une famille de jeux de cartes G-couvrants sur~$X$
qui est filtrante pour la relation d'ordre~$\preceq$:
si $\mathscr D$ et $\mathscr D'$ appartiennent à~$\Gamma$,
il existe un jeu de cartes $\mathscr E$ dans~$\Gamma$
tel que $\mathscr D\preceq \mathscr E$ et $\mathscr D'\preceq \mathscr E$.

Les familles suivantes sont les principales familles de jeux de cartes
que nous considérerons:
\begin{enumerate}
\item 
L'ensemble de tous les jeux cartes G-couvrants sur~$X$.
Vérifions qu'il est filtrant. Soit $\mathscr D$ et $\mathscr D'$
des jeux de cartes G-couvrants sur~$X$.
Soit $\mathscr E$ la famille des cartes $(U\cap U', (f|_{U\cap U'},f'|_{U\cap U'}))$, où $(U,f)$ et $(U',f')$ parcourent les cartes de~$\mathscr D$
et~$\mathscr D'$ respectivement. C'est un jeu G-couvrant,
et il raffine~$\mathscr D$ et~$\mathscr D'$ par construction.

\item Supposons que $X$ soit un sous-espace paralinéaire fermée de~$\R^n$.
On peut prendre pour famille $\Gamma$ l'ensemble réduit 
à la carte identique de~$X$.

\item Lorsque $X$ est un espace analytique, nous prendrons pour
famille~$\Gamma$ les jeux de cartes $\mathscr D$
tels que les intérieurs des domaines des cartes de~$\mathscr D$
recouvrent~$X$ (un tel jeu est G-couvrant)
et dont les applications de cartes sont de la forme $x\mapsto
(\log (\abs{f_1}(x)), \dots, \log(\abs{f_n})(x))$ où les~$f_i$
sont des fonctions analytiques inversibles sur le domaine de la carte.

\item Lorsque $X$ est un espace analytique strict, 
on pourrait considérer les jeux G-couvrants
dont les domaines de définition sont  des domaines stricts.
\end{enumerate}

Soit $u\colon X\to Y$ un morphisme d'espaces G-tropicaux.
Soit $\Gamma$ une famille filtrante
de jeux de cartes G-couvrants sur~$X$,
soit $\Delta$ une famille filtrante de jeux de cartes
G-couvrants sur~$Y$.
On dit que $\Gamma\preceq_u \Delta$ si 
pour tout jeu de cartes $\mathscr D\in\Gamma$,
il existe un jeu de cartes $\mathscr E\in\Delta$
tel que $\mathscr D\preceq_u \mathscr E$.

\subsection{}
Soit $\Gamma$ une famille filtrante de jeux de cartes G-couvrants sur~$X$
et soit $\Phi$ une famille de supports sur~$X$.
On appelle $\Gamma$-$\Phi$-forme sur~$X$ 
une G-forme dont le support est un élément de~$\Phi$
et qui est G-tropicalisée par un jeu de cartes appartenant à~$\Gamma$.

On note $\mathscr A^{p,q}_{\Gamma,\Phi}$ l'ensemble 
des $\Gamma$-$\Phi$-formes de type~$(p,q)$ sur~$X$.
C'est la réunion des sous-espaces $\mathscr A_{\groth,\mathscr D,F}^{p,q}(X)$
des G-formes sur~$X$ dont le support est contenu dans
une partie fermée $F\in\Phi$ et qui sont G-tropicalisées 
par le jeu~$\mathscr D\in\Gamma$.
Pour tout élément $m\in\N\cup\{+\infty\}$,
on munit chacun des ces espaces $\mathscr A_{\groth,\mathscr D,F}^{p,q}(X)$
de la topologie induite par la topologie~$\mathscr C^m$ 
de $\mathscr A_{\groth,\mathscr D}^{p,q}(X)$, et leur réunion $\mathscr A^{p,q}_{\Gamma,\Phi}$ de la topologie localement convexe colimite.
On dit que c'est la topologie~$\mathscr C^m$
sur $\mathscr A^{p,q}_{\Gamma,\Phi}$.

Par définition, une application linéaire $T$ sur~$\mathscr A^{p,q}_{\Gamma,\Phi} $, à valeurs dans un espace vectoriel topologique localement convexe~$E$,
est continue pour la topologie~$\mathscr C^m$
si et seulement si chacune de ses restrictions à 
un sous-espace $\mathscr A_{\groth,\mathscr D,F}^{p,q}(X)$ est continue
pour sa topologie~$\mathscr C^m$.

Si $m\leq n$, la topologie~$\mathscr C^m$
sur $\mathscr A^{p,q}_{\Gamma,\Phi}$
est moins fine que la topologie~$\mathscr C^n$.

\subsection{}
Soit $\Gamma$ et $\Delta$ deux familles filtrantes
de jeux de cartes G-couvrants sur~$X$ tels que $\Gamma\preceq\Delta$.
(Cela équivaut à ce que $\Delta$ soit cofinal dans
l'ensemble $\Gamma\cup\Delta$, et entraîne que ce dernier ensemble 
est filtrant.)
Soit $\Phi$ et $\Psi$ des familles de supports
telles que $\Phi\subset\Psi$.
Alors, toute $\Gamma$-forme sur~$X$ à support
dans un fermé de~$\Phi$ est également une $\Delta$-forme
sur~$X$ à support dans un fermé de~$\Psi$,
et il résulte de la définition des topologies
en jeu l'inclusion $\mathscr A^{p,q}_{\Gamma,\Phi}\subset \mathscr A^{p,q}_{\Delta,\Psi}$ est continue pour les topologies~$\mathscr C^m$.

Lorsque les familles $\Gamma$ et $\Delta$ sont simultanément cofinales
dans leur réunion, il en résulte que les deux
espaces de formes $\mathscr A^{p,q}_{\Gamma,\Phi}$ et $\mathscr A^{p,q}_{\Delta,\Phi}$ coïncident, avec leurs topologies~$\mathscr C^m$.

\begin{remas}
\begin{enumerate}
\item
Grâce à la théorie de l'intégration des formes qui fournira
des formes linéaires non triviales,
on verra plus loin des conditions  raisonnables qui assurent 
que la topologie de~$\mathscr A^{p,q}_{\Gamma,\Phi}(X)$ est séparée.

\item
Si $\Phi$ est l'ensemble des parties fermées de~$X$
contenues dans une partie fermée~$F$,
alors $\mathscr A^{p,q}_{\groth,\Phi}(X)=\mathscr A^{p,q}_{\groth,F}(X)$
et les topologies coïncident.

De même, si $X$ est compact, l'espace
$\mathscr A^{p,q}_{\Gamma,\cpct}(X)=\mathscr A^{p,q}_{\Gamma}(X)$
et leurs topologies coïncident.

\item
Conformément à l'usage en théorie des distributions,
nous imposons à la topologie des espaces~$\mathscr A_{\Gamma,\Phi}^{p,q}(X)$
d'être localement convexe. Cependant, 
les considérations de ce travail ne font essentiellement intervenir
que des applications linéaires à valeurs dans des espaces localement
convexes (notamment~$\R$!) ou dans 
des niveaux finis des systèmes inductifs en jeu,
de sorte que cette subtilité reste invisible.
\end{enumerate}
\end{remas}

\begin{prop}\label{prop.continuite-wedge}
Soit $\Gamma$ une famille filtrante de jeux de cartes G-couvrants sur~$X$
et soit $\alpha\in\mathscr A^{p',q'}_{\Gamma}$.
Soit $\Phi$ et $\Psi$ des familles de supports sur~$X$;
on suppose que pour toute partie fermée~$F$ de~$\supp(\alpha)$
appartenant à~$\Phi$ appartient également à~$\Psi$.

Pour toute forme $\omega\in\mathscr A^{p,q}_{\Gamma,\Phi}(X)$,
la forme $\alpha\wedge\omega$ appartient à $\mathscr A^{p,q}_{\Gamma,\Psi}(X)$
et l'application 
\[   A^{p,q}_{\Gamma,\Phi}(X) \to \mathscr A^{p+p',q+q'}_{\Gamma,\Psi}(X), 
\qquad \omega\mapsto \alpha\wedge\omega \]
est linéaire continue.
\end{prop}
\begin{proof}
Soit $\mathscr E\in\Gamma$ un jeu de cartes qui G-tropicalise~$\alpha$,
et soit $\mathscr D$ un jeu de cartes qui G-tropicalise~$\omega$;
il existe un jeu de cartes~$\mathscr D'\in\Gamma$
tel que $\mathscr D'\preceq \mathscr D$ et $\mathscr D'\preceq E$.
Alors, $\alpha$ et $\omega$ sont G-tropicalisées par~$\mathscr D'$,
de sorte que leur produit l'est aussi.
Son support est contenu dans $\supp(\alpha)\cap\supp(\omega)$,
donc appartient à~$\Psi$.

La linéarité de l'application donnée est évidente.

Par définition de la topologie~$\mathscr C^m$ sur $\mathscr A^{p,q}_{\Gamma,\Phi}(X) $ (le membre de droite étant localement convexe),
il suffit de prouver que pour tout jeu de cartes $\mathscr D\in\Gamma$,
l'application  restreinte
\[   A^{p,q}_{\mathscr D,\Phi}(X) \to \mathscr A^{p+p',q+q'}_{\Gamma,\Psi}(X) \]
est continue. 
Soit $\mathscr D'\in\Gamma$ un jeu de cartes 
tel que $\mathscr D'\preceq \mathscr D$ et $\mathscr D'\preceq E$;
cette application linéaire se factorise alors
par $\mathscr A^{p+p',q+q'}_{\mathscr D',\Psi}(X)$.
D'après la proposition~\ref{prop.continuite-wedge-jeu}, 
l'application linéaire
\[   A^{p,q}_{\mathscr D}(X) \to \mathscr A^{p+p',q+q'}_{\mathscr D'}(X) \]
est continue. L'assertion en découle.
\end{proof}

\begin{prop}\label{prop.continuite-pull-back}
Soit $u\colon Y\to X$ un morphisme d'espaces G-tropicaux. 
Soit $\Phi$ une famille de supports sur~$X$ et 
soit~$\Psi$ une famille de supports sur~$Y$;
on suppose que pour tout fermé $F$ de $X$ appartenant
à~$\Phi$, la partie fermée $u^{-1}(F)$ de~$Y$ appartient à~$\Psi$.
Soit $\Gamma$ une famille filtrante de jeux de cartes G-couvrants sur~$X$
et soit $\Delta$ une famille filtrante de jeux de cartes G-couvrants sur~$Y$
telles que $\Delta\preceq_u \Gamma$.

Soit $(p,q)$ un couple d'entiers naturels et 
soit $m\in\N\cup\{\infty\}$.

Pour toute G-forme $\alpha\in\mathscr A^{p,q}_{\Gamma,\Phi}(X)$,
la G-forme $u^*\alpha$ appartient à~$\mathscr A^{p,q}_{\Delta,\Psi}(Y)$.
L'application
\[ u^*\colon \mathscr A^{p,q}_{\Gamma,\Phi}(X) \to
\mathscr A^{p,q}_{\Delta,\Psi}(Y), \qquad
\alpha\mapsto u^*\alpha \]
est linéaire et continue.
\end{prop}
\begin{proof}
Soit $\mathscr D$ un jeu de cartes appartenant à~$\Gamma$
tel que $\alpha$ soit G-tropicalisée par~$\mathscr D$.
Par hypothèse, il existe un jeu de cartes $\mathscr E$
appartenant à~$\Delta$ tel que $\mathscr E\preceq_f \mathscr D$.
Alors, $u^*\alpha$ est G-tropicalisée par~$\mathscr E$.
Son support est contenu dans $u^{-1}(\supp(\alpha))$,
donc il appartient à~$\Psi$.
Par suite, $u^*\alpha\in\mathscr A^{p,q}_{\Delta,\Psi}(Y)$.

L'application $u^*$ est évidemment linéaire.

Pour vérifier sa continuité, la définition
des topologies~$\mathscr C^m$ en jeu montre qu'il suffit d'observer 
que pour tout couple de jeux de cartes $(\mathscr D,\mathscr E)$
comme ci-dessus, 
l'application
$ u^*\colon \mathscr A^{p,q}_{\mathscr D}(X) \to
\mathscr A^{p,q}_{\mathscr E}(Y)$
est continue.
Ce fait a été démontré dans la proposition~\ref{prop.raffine-topologie}.
\end{proof}

\begin{coro}\label{coro.continuite-pull-back}
Soit $f\colon Y\to X$ un morphisme d'espaces G-tropicaux. 
Soit $m\in\N\cup\{+\infty\}$.

\begin{enumerate}
\item L'application $f^*\colon \mathscr A^{p,q}_\groth(X)
           \to\mathscr A^{p,q}_\groth(Y)$ est continue
pour les topologies~$\mathscr C^m$ sur ces espaces;
\item Si le morphisme~$f$ est \emph{compact},
l'application $f^*\colon \mathscr A^{p,q}_{\groth,\cpct }(X)
           \to\mathscr A^{p,q}_{\groth,\cpct }(Y)$ est continue
pour les topologies~$\mathscr C^m$ sur ces espaces.
\end{enumerate}
\end{coro}

\subsection{}
Soit $X$ un espace G-tropical et soit $U$ un ouvert de~$X$.

Soit $\Phi$ une famille de supports sur~$U$,
soit $\Psi$ une famille de supports sur~$X$; 
on suppose que pour toute partie $F\in\Phi$, 
la partie~$F$ est fermée dans~$X$ et appartient à~$\Psi$.
(En particulier, on peut prendre pour $\Phi$ et~$\Psi$
les ensembles de parties compactes de~$U$ et~$X$ respectivement.)
 
Soit $\Gamma$ une famille filtrante de jeux de cartes de cartes
G-couvrants sur~$U$ et soit $\Delta$ une famille filtrante de jeux 
de cartes G-couvrants sur~$X$.
On suppose que pour tout $\mathscr D\in\Gamma$
et toute partie $F\in\Phi$, il existe $\mathscr E\in\Delta$
tel que $\mathscr E\preceq \mathscr D\cup\{\mathbf 0 \colon X\setminus F\to \R^0\}$.

Si $\Gamma$ et $\Phi$ sont donnés, 
on peut par exemple prendre pour famille~$\Delta$
la famille~$\Gamma^\Phi$ des jeux de cartes de la forme $\mathscr D\cup\{\mathbf 0 \colon X\setminus F\to \R^0\}$, où $\mathscr D\in\Gamma$ et $F\in\Phi$.
Notons que l'on a alors $\Gamma\preceq_j \Gamma^\Phi$.

Dans ces conditions, toute G-forme~$\omega\in\mathscr A^{p,q}_{\Gamma,\Phi}(U)$
admet un prolongement par~$0$, noté $j_!\omega$, caractérisée par 
les relations $j_!\omega|_{U}=\omega$
et $j_!\omega|_{X\setminus\supp(\omega)}=0$.
(Noter que $U$ et $X\setminus\supp(\omega)$ forment un recouvrement
ouverts de~$X$.)

Observons que l'opération $j_!$  commute aux opérateurs~$\mathrm J$,
$\di$ et~$\dc$.

\begin{prop}\label{prop.continuite-push-U}
Conservons les notations ci-dessus.
\begin{enumerate}
\item Pour toute G-forme~$\omega\in\mathscr A^{p,q}_{\Gamma,\Phi}(U)$,
la G-forme $j_!\omega$ sur~$X$ appartient à $\mathscr A^{p,q}_{\Delta,\Psi}(X)$.
et l'application 
\[ j_!\colon\mathscr A^{p,q}_{\Gamma, \Phi }(U)\to \mathscr A^{p,q}_{\Delta, \Psi }(X) \]
ainsi définie 
est linéaire et continue pour les topologies~$\mathscr C^m$.

\item
Cette application est injective.

\item
Supposons que $\Gamma\preceq_j \Delta$ et que toute
partie fermée~$F$ de~$X$ appartenant à~$\Psi$ qui contenue dans~$U$
appartient à~$\Phi$.
Supposons que plus que $d_\trop(x)<\max(p,q)$
pour tout $x\in X\setminus U$.
Alors cette application~$j_!$ est un homéomorphisme, d'inverse donné par~$j^*$.
\end{enumerate}
\end{prop}
\begin{proof}
\begin{enumerate} 
\item
Soit $\omega\in\mathscr A^{p,q}_{\Gamma,\Phi}(U)$;
soit $F$ son support et soit $\mathscr D\in\Gamma$ un jeu de cartes
qui la G-tropicalise.
Il est clair que la G-forme $j_!\omega$ est
G-tropicalisée par le jeu de cartes~$\mathscr D^F$ 
obtenu en adjoignant à~$\mathscr D$
la carte constante sur $X\setminus F$.
Son support est égal à~$F$, donc appartient à~$\Psi$.
Ainsi, $j_!\omega\in\mathscr A^{p,q}_{\Delta,\Psi}(X)$.
Il est évident que l'application~$j_!$ ainsi définie est linéaire.

Démontrons qu'elle est continue.
Par définition des topologies~$\mathscr C^m$ en jeu,
il suffit de démontrer que pour tout jeu de cartes~$\mathscr D\in\Gamma$
et toute partie fermée $F\in\Phi$,
l'application
\[ \mathscr A^{p,q}_{\mathscr D,F}(U) \to \mathscr A^{p,q}_{\Delta,\Psi}(X) \]
induite par restriction est continue.
Cette application se factorise par $\mathscr A^{p,q}_{\mathscr D^F, F}(X)$
si bien qu'il suffit de démontrer que l'application 
\[ \mathscr A^{p,q}_{\mathscr D,F}(U) \to \mathscr A^{p,q}_{\mathscr D^F,F}(X) \]
déduite de~$j_!$ est elle-même continue pour les topologies~$\mathscr C^m$.

Soit $g$ une carte appartenant à~$\mathscr D^F$
et soit $W$ une domaine compact de son domaine de définition~$V$.
Il y a deux cas à considérer, suivant que $g$ est une carte
de~$\mathscr D$ ou bien $g$ est la carte constante de domaine~$X\setminus F$.
Dans ce dernier cas, les semi-normes associées à~$g$ et au domaine
compact~$W$ sont nulles.
Dans le premier cas, les semi-normes associées à~$g$ et au domaine
compact~$W$ sont les mêmes que l'on considère cette carte comme
carte sur~$U$ ou sur~$X$.  La propriété de continuité s'en déduit.

\item
Comme $j^*(j_!\omega)=\omega$, l'application~$j_!$
est injective. 

\item
Soit $\alpha\in\mathscr A^{p,q}_{\Delta,\Psi}(X)$, soit $S$ son support.
Comme $S$ est contenu dans l'ensemble des points~$x$ de~$X$
tels que $d_\trop(x)\geq \max(p,q)$ 
(proposition~\ref{lemm.support-comp-ferme}), on a $S\subset U$.
Par suite, $\alpha|_U$ est à support dans~$\Phi$
et l'on a $\alpha=j_!(\alpha|_U)$.
Cela démontre que l'application~$j_!$ est bijective.
Sa réciproque est l'application~$j^*$, donc est continue
d'après la proposition~\ref{prop.continuite-pull-back}.
\qedhere\end{enumerate}
\end{proof}

\begin{prop}\label{prop.continuite-di-dc-J}
Soit $X$ un espace G-tropical.
Soit $\Gamma$ une famille filtrante de jeux de cartes G-couvrants sur~$X$,
soit $\Phi$ une famille de supports sur~$X$, 
et soit $m\in\N\cup\{+\infty\}$.
\begin{enumerate}
\item Pour toute G-forme $\omega\in\mathscr A^{p,q}_{\Gamma,\Phi}(X)$,
la forme $\mathrm J\omega$ appartient à~$\mathscr A^{q,p}_{\Gamma,\Phi}(X)$. 
L'application~$\mathrm J$ ainsi définie est linéaire et continue.
\item Supposons $m\geq 1$.
Pour toute G-forme $\omega\in\mathscr A^{p,q}_{\Gamma,\Phi}(X)$,
on a $\di\omega \in\mathscr A^{p+1,q}_{\Gamma,\Phi}(X)$
et $\dc\omega \in\mathscr A^{p,q+1}_{\Gamma,\Phi}(X)$.
Les applications~$\di$ et~$\dc$ ainsi définies sont continues
lorsqu'on munit $\mathscr A^{p,q}_{\Gamma,\Phi}(X)$
de la topologie~$\mathscr C^m$ et les
espaces $\mathscr A^{p+1,q}_{\Gamma,\Phi}(X)$
et $\mathscr A^{p,q+1}_{\Gamma,\Phi}(X)$
de la topologie~$\mathscr C^{m-1}$.
\end{enumerate}
\end{prop}
\begin{proof}
Compte tenu de la définition des topologies en jeu, l'assertion se ramène
à l'assertion analogue
à jeu de cartes fixé, qui est 
la proposition~\ref{prop.continuite-di-dc-jeu}.
\end{proof}

%
%
%

\chapter{Calibrages et intégration sur les espaces paralinéaires}

\section{Négligeabilité paralinéaire}

Soit $X$ un espace paralinéaire.

\subsection{}
Soit $d$ un entier~$\geq 0$; 
notons $\mathscr I_d$
le sous-préfaisceau de~$\mathscr A_{X_\groth}$ donné par
\[ \mathscr I_d(Y) = \ker( \mathscr A_{X_\groth}(Y) \to \mathscr A_{X_\groth}(Y^{(d)})), \]
pour tout sous-espace paralinéaire~$Y$ de~$X$,
où $Y^{(d)}$ désigne l'ensemble des points~$y\in Y$
tels que $\dim_y(Y)\geq d$.

C'est un faisceau. Soit en effet $Y$ un sous-espace paralinéaire
de~$X$, soit $(Y_i)_{i\in I}$ un G-recouvrement de~$Y$
et soit $f\in\mathscr A_{X_\groth}(Y)$ une fonction G-lisse dont la restriction
à $Y_i^{(d)}$ est nulle pour tout~$i$. Soit $x\in Y^{(d)}$;
par définition d'un G-recouvrement, il existe une partie finie~$J$
de~$I$ telle que $x\in Y_i$ pour tout $i\in J$,
et telle que $\bigcup_{i\in J} Y_i$ soit un voisinage de~$x$.
Pour tout~$i\in J$, soit $V_i$ un voisinage paralinéaire
compact de~$x$ dans~$Y_i$ tel que $\dim(V_i)=\dim_x(Y_i)$.
La réunion des~$V_i$ est un sous-espace paralinéaire
compact de~$X$, de dimension $\max_{i\in J} (\dim_x(Y_i))$.
Comme c'est un  voisinage de~$x$, 
il existe $i\in J$ tel que $\dim_x(Y_i)\geq d$. 
Alors, $f(x)=f|_{Y_i}(x)=0$.

\begin{lemm}\label{lemm.noyau}
Soit $\mathscr M$ un $\mathscr A_{X_\groth}$-module.
Soit $Y$ un sous-espace
paralinéaire de~$X$.
On a $\mathscr I_d \mathscr M (Y)= \ker(\mathscr M(Y)\to\mathscr M(Y^{(d)}))$.
\end{lemm}
\begin{proof}
L'inclusion $\mathscr I_d\mathscr M(Y)\subset \ker(\mathscr M(Y)\to\mathscr M(Y^{(d)}))$ est évidente.
Il suffit de vérifier l'autre inclusion G-localement; on peut donc supposer
que $Y$ est une cellule.
Soit $\omega\in\ker(\mathscr M(Y)\to\mathscr M(Y^{(d)}))$.
Si $\dim(Y)\geq d$, alors $Y=Y^{(d)}$ et $\omega=0$,
donc $\omega\in\mathscr I_d\mathscr M(Y)$.
Si $\dim(Y)<d$, alors $Y^{(d)}=\emptyset$, donc $\mathscr I_d(Y)$ contient~$1$,
d'où $\omega=1\cdot\omega\in \mathscr I_d\mathscr M(Y)$.
\end{proof}

\begin{coro}
Soit $\mathscr M$ et $\mathscr M'$ des $\mathscr A_{X_\groth}$-modules
tels que 
$\mathscr M\subset\mathscr M'$.
L'injection de $\mathscr M$ 
dans~$\mathscr M'$
induit une injection de $\mathscr M/\mathscr I_d\mathscr M$
dans 
$\mathscr M'/\mathscr I_d\mathscr M'$.
\end{coro}

\begin{coro}\label{prop.A(p,q)-torsion}
Soit $\mathscr G$ une classe de coefficients.
Si $d=\max(p,q)$, 
le faisceau d'idéaux~$\mathscr I_d\mathscr G_{X_\groth}$ 
est l'annulateur du faisceau $\mathscr G_{X_\groth}^{p,q}$.
\end{coro}
\begin{proof}
Soit $Y$ un sous-espace paralinéaire de~$X$.

Soit $f\in\mathscr I_d\mathscr G(Y)$ 
et soit $\omega\in\mathscr G_{X_\groth}^{p,q}(Y)$.
Soit $(Y_i)$ un G-recouvrement de~$Y$ par des cellules.
Si $\dim(Y_i)<d$, alors $\omega|_{Y_i}=0$;
sinon, $f|_{Y_i}=0$. Dans les deux cas, $f\omega|_{Y_i}=0$.
On a donc $f\omega=0$.

Inversement, soit $f\in\mathscr G_{X_\groth}(Y)$ 
et supposons que $f$
annule la restriction à~$Y$ du faisceau $\mathscr G^{p,q}_{X_\groth}$;
démontrons que $f\in\mathscr I_d\mathscr G(Y)$.
D'après le lemme~\ref{lemm.noyau},
il suffit de prouver que la restriction de~$f$ à toute
cellule de dimension~$d$ de~$Y$ est nulle; soit $C$ une telle cellule
et soit $j\colon C\to \R^d$ une immersion paralinéaire
qui induit un isomorphisme de~$C$ sur un sous-espace paralinéaire 
de~$\R^d$. Soit $u$ l'unique fonction 
sur~$j(C)$ telle que $f|_C=j^*u$. 
Quitte à considérer un G-recouvrement de la cellule~$C$,
on peut supposer que $u$ est la restriction 
d'une fonction de classe~$\mathscr G$  sur un voisinage de~$j(C)$.
Fixons alors un couple $(I,J)$ de multi-indices de longueurs~$(p,q)$
dans~$\{1,\dots,d\}$.
On  a 
\[ 
f|_C j^*(\di x_I\wedge\dc x_J) 
 =  j^*(u \di x_I\wedge\dc x_J)=0, 
\]
donc $u\di x_I\wedge\dc x_J=0$. Il découle alors de l'exemple~\ref{exem.formes-cellules-valeurs} que $u=0$, d'où \mbox{$f|_C=0$.}
\end{proof}

%

\begin{rema}\label{rema.modIn}
Soit $X$ un espace paralinéaire de dimension~$\leq n$
et soit $\mathscr M$ un $\mathscr A_{X_\groth}$-module.
Posons $\mathscr N=\mathscr M/\mathscr I_n\mathscr M$; 
décrivons ses sections lorsque $X$ est paracompact.

Soit $\mathscr C$ une décomposition cellulaire de~$X$ et soit
$(f_C)$ une famille de sections de~$\mathscr M$, où $f_C\in\mathscr M(C)$
pour tout $C\in\mathscr C$.
Si $C$ et~$D$ sont deux cellules distinctes appartenant à~$\mathscr C$,
leur intersection est de dimension~$<n$, de sorte que les restrictions
de~$f_C$ et~$f_D$ à~$C\cap D$ coïncident modulo~$\mathscr I_n$.
La famille $(f_C)$ définit donc une section~$f$ de~$\mathscr N(X)$.

Comme $X$ est paracompact, toute section de~$\mathscr N$ peut se décrire
ainsi.

Soit enfin deux sections~$g$ et~$g'$ de~$\mathscr N$.
Choisissons une décomposition cellulaire~$\mathscr C$ et deux familles $(g_C)$
et $(g'_C)$ de sections de~$\mathscr M$
définissant respectivement~$g$ et~$g'$; il en existe, car $X$
est paracompact. 
Alors, $g=g'$ si et seulement si $g_C=g'_C$ modulo~$\mathscr I_n$
pour toute cellule~$C\in\mathscr C$,
c'est-à-dire $g_C=g'_C$ pour tout $C\in\mathscr C_n$.
\end{rema}

\subsection{}
Soit $X$ un espace paralinéaire de dimension~$\leq n$.
On dit qu'une section~$f$ 
de~$\mathscr F_{X_\groth}/\mathscr I_n\mathscr F_{X_\groth}$
est positive (resp. strictement positif, resp. G-localement constante)
si, G-localement, elle admet un représentant dans~$\mathscr F_{X_\groth}$
qui est positif (resp. strictement positif, resp. constant).

\begin{prop}\label{prop.A(n,q)-libre}
Soit $\mathscr G$ une classe de coefficients.
Soit $X$ un espace paralinéaire paracompact de dimension~$\leq n$
et  soit $(p,q)$ un couple d'entiers naturels tels que
$0\leq p,q\leq n$.
Le G-faisceau $\mathscr G^{p,q}_{X_\groth}/\mathscr I_n\mathscr G^{p,q}_{X_\groth}$
est libre de rang $\binom np\binom nq$ 
sur l'anneau $\mathscr G_{X_\groth}/\mathscr I_n\mathscr G_{X_\groth}$.
\end{prop}
\begin{proof}
Soit $\mathscr C$ une décomposition cellulaire de~$X$.
Pour chaque cellule $C\in\mathscr C$ telle que $\dim(C)=n$,
choisissons une immersion paralinéaire~$j_C$ de~$C$
dans~$\R^n$ dont l'image est une cellule.
Pour tout multi-indice~$I$ de longueur~$p$ dans~$\{1,\dots,n\}$
et tout multi-indice~$J$ de longueur~$q$ dans~$\{1,\dots,n\}$,
posons $\omega_{I,J,C}=j_C^*(\di x_I \wedge \dc x_J)$.
Si $C$ est une cellule de dimension~$<n$, posons
$\omega_{I,J,C}=0$.
Il existe une unique section $\omega_J\in\mathscr A^{p,q}_{X_\groth}/\mathscr I_n \mathscr A^{p,q}_{X_\groth}(X)$
dont la restriction à toute cellule~$C\in\mathscr C$ est égale à~$\omega_{I,J,C}$.
En effet, les conditions de recollement 
dans le faisceau quotient sont triviales,
puisque l'intersection de deux cellules distinctes 
est de dimension~$\leq n-1$.
Démontrons que la famille $(\omega_J)$ est une base du
G-faisceau $\mathscr G^{p,q}_{X_\groth}/\mathscr I_n\mathscr G^{p,q}_{X_\groth}$ sur l'anneau
$\mathscr G_{X_\groth}/\mathscr I_n\mathscr G_{X_\groth}$.

Elle est libre. En effet, si $(f_{I,J})$ est une famille
d'éléments de $\mathscr G_{X_\groth}(X)$ telle que $\sum f_{I,J}\omega_{I,J}=0$,
on a $\sum f_{I,J}|_C \omega_{I,J,C}=0$ pour toute cellule $C\in\mathscr C$.
Si $C$ est de dimension~$n$, l'argument utilisé pour déterminer
la torsion de $\mathscr G^{n,q}_{X_\groth}$
entraîne que $f_J|_C=0$. Sinon, $\dim(C)<n$,
donc $f_J|_C\in\mathscr I_{n}(C)$. 
On a ainsi $f_J\in\mathscr I_n\mathscr G_{X_\groth}(X)$
pour tout~$J$, d'où l'assertion.

Démontrons maintenant que cette famille est génératrice.
Soit $\omega\in\mathscr G^{p,q}_{X_\groth}/\mathscr I_n\mathscr G^{p,q}_{X_\groth}(X)$.
Soit $C$ une cellule de~$\mathscr C$.
Si $\dim(C)<n$, on pose $f_{I,J,C}=0$ pour tous~$I,J$.
Supposons maintenant $\dim(C)=n$.
Il existe une décomposition cellulaire~$\mathscr D_C$ de~$j_C(C)$
et, pour toute cellule~$D$ de~$\mathscr D_C$, un élément
$\omega_D$  de~$\mathscr A^{p,q}_{D}(D)$ 
tel que $\omega|_{j_C^{-1}(D)} =j_C^* \omega_D$.
D'après l'exemple~\ref{exem.formes-cellules-valeurs},
il existe une unique famille $(g_{I,J,D})$ de fonctions  de classe~$\mathscr G$
au voisinage de~$D$ telles que 
$\omega_{D}=\sum g_{I,J,D} \di x_I\wedge\dc x_J$.
La famille $(g_{I,J,D})$, où $D$ parcourt~$\mathscr D_C$,
se recolle en un élément~$f_{I,J,C}$ de~$\mathscr G_{X_\groth}/\mathscr I_n\mathscr G_{X_\groth}(C)$;
les intersections deux à deux étant de dimension~$\leq n-1$.
Pour la même raison, la famille $(f_{I,J,C})$ se recolle
en une section $f_{I,J}$ 
de~$\mathscr G_{X_\groth}/\mathscr I_n\mathscr G_{X_\groth}(X)$
et l'on a $\omega=\sum f_{I,J}\omega_{I,J}$.
\end{proof}

\begin{rema}\label{rema.changement-base-A/N}
Soit $\mathscr G$ une classe de coefficients.

Soit $f$ une immersion par morceaux de~$X$ dans~$\R^n$. 
Pour tout sous-ensemble~$I$ de cardinal~$p$ de~$\{1,\dots,n\}$
et tout sous-ensemble~$J$ de cardinal~$q$  de~$\{1,\dots,n\}$,
notons  $\alpha_{I,J}$ 
la forme de type~$(p,q)$ sur~$\R^n$ donnée par 
\[ \alpha_J = \di x_I\wedge\dc x_J. \]
Il résulte de la  preuve  de la proposition~\ref{prop.A(n,q)-libre}
que la famille $(f^*\alpha_{I,J})_J$ 
est une base du faisceau~$\mathscr G^{p,q}_{X_\groth}$
comme $(\mathscr G_{X_\groth}/\mathscr I_n\mathscr G_{X_\groth})$-module. 
On dira que c'est la base
canonique de ce faisceau associée à l'immersion par morceaux~$f$.

Soit $g\colon X\to \R^n$ une seconde immersion par morceaux.
La famille $(g^*\alpha_{I,J})_{I,J}$
est aussi une base du faisceau~$\mathscr G^{p,q}_{X_\groth}$.
Soit $C$ une cellule de dimension~$n$ de~$X$ telle
$(f,g)(C)$ soit contenu dans un sous-espace affine de dimension~$n$
de~$\R^{2n}$. Soit $u$ l'unique application affine
de~$\R^n$ dans lui-même tel que $g|_C=u\circ f|_C$.
Pour tout couple~$(I,J)$ de sous-ensembles de~$\{1,\dots,n\}$
tels que $\Card(I)=p$ et $\Card(J)=q$, on a donc
\[ g^*\alpha_{I,J} = \sum_K u_{IK} \sum_{L}  u_{JL} 
f^*\alpha_{KL}, \]
où $u_{IK}$ (resp. $u_{JL}$) est le déterminant de la matrice déduite de~$u$
par extraction des lignes d'indices dans~$I$ et des colonnes d'indices dans~$K$
(resp. des lignes d'indices dans~$J$ et des colonnes d'indices dans~$L$).

Prenons $\mathscr G=\mathscr A$.
On voit en particulier que la matrice de changement
de base de la famille~$(f^*\alpha_{I,J})$ à la famille $(g^*\alpha_{K,L})$
est une section G-localement constante
du G-faisceau $\mathrm M_{\binom np\binom nq} (\mathscr A_{X_\groth}/\mathscr I_n)$.

Lorsque $p=n$ et $q=0$, $\binom np\binom nq=1$, et cette matrice est réduite
à une « fonction »~$J_{g/f}$ qu'on appellera \emph{jacobien relatif}
des deux immersions par morceaux~$f$ et~$g$.
C'est une section de~$\mathscr A_{X_\groth}/\mathscr I_n$
qui est G-localement constante et inversible.
Sur une cellule~$C$ de dimension~$n$ comme ci-dessus, elle est effectivement
donnée par le jacobien de l'automorphisme affine~$u$ de~$\R^n$
tel que $g|_C=u\circ f|_C$.
\end{rema}

\begin{defi}\label{defi.loc-constant}
Soit $\mathscr G$ une classe de coefficients.
On dira qu'une section~$\omega$ de $\mathscr G^{p,q}_{X_\groth}$
est \emph{G-localement constante} si pour tout
sous-espace paralinéaire~$Y$ de~$X$ et toute
immersion par morceaux $f\colon Y\to\R^n$, les coefficients
de~$\omega|_Y$ dans la base canonique $(f^*\omega_{I,J})$ 
associée à~$f$ sont des sections
G-localement constantes 
de~$(\mathscr G_{X_\groth}/\mathscr I_n\mathscr G_{X_\groth})(Y)$.
\end{defi}

D'après ce qui précède, une section~$\omega$ est 
G-localement constante si et seulement s'il existe un G-recouvrement~$(X_i)$
de~$X$ et, pour tout~$i$, une immersion par morceaux $f_i\colon X_i\to\R^n$,
tels que pour tout~$i$, les coefficients de~$\omega|_{X_i}$
dans la base~$(f_i^*\omega_{I,J})_{I,J}$  soient G-localement constants.\index{forme de type $(n,q)$ G-localement constante}

On définit de manière analogue, en considérant les bases
duales des précédentes, la notion d'une section
G-localement constante du  faisceau
\[ \ShHom_{X_\groth}
  (\mathscr G^{p,q}_{X_\groth}, (\mathscr G_{X_\groth}/\mathscr I_n\mathscr G_{X_\groth})). \]
Ses sections sont celles qui appliquent une $(p,q)$-forme  
G-localement
constante sur une section G-localement constante 
de $\mathscr G_{X_\groth}/\mathscr I_n\mathscr G_{X_\groth}$.

\subsection{}
Soit $f\colon X \to Y$ une immersion par morceaux compacte
d'espaces paralinéaires de dimensions~$\leq n$.
Soit $\mathscr G$ une classe de coefficients
et soit $(p,q)$ un couple d'entiers naturels tel que $\max(p,q)=n$.
Nous allons construire un morphisme de faisceaux
\[ f_*\colon  f_*{\mathscr G^{p,q}_{X_\groth}}
 \to  {\mathscr G^{p,q}_{Y_\groth}}. \]

Soit $U$ un sous-espace paralinéaire paracompact de~$Y$.
Alors, $f_U\colon f^{-1}(U)\to U$ est une immersion par morceaux
compacte et $f^{-1}(U)$ est paracompact.
D'après la proposition~\ref{prop.decomp-relative},
il existe un couple $(\mathscr C,\mathscr D)$ 
de décompositions cellulaires qui est adapté à~$f$.
Soit $\omega$ une $(p,q)$-forme sur $f^{-1}(U)$.
Soit $D$ une cellule de~$\mathscr D$ telle que $\dim(D)=n$.
L'ensemble des cellules $C\in\mathscr C_n$ telles que $f(C)=D$
est fini; pour toute telle cellule~$C$, $f|_C$ induit un isomorphisme
de~$C$ sur~$D$, de sorte que $\omega|_C$ induit une $(p,q)$-forme sur~$D$.
Notons $f_*\omega_D$ la somme de ces $(p,q)$-formes.

Comme $\max(p,q)=n$, il existe une unique $(p,q)$-forme $f_*\omega$ sur~$U$ 
dont la restriction à toute $n$-cellule~$D\in\mathscr D$
est égale à~$f_*\omega_D$: sa restriction à une cellule de~$\mathscr D$
qui n'est pas de dimension~$n$ est nulle, et les conditions
de recollement sont automatiques.

Cette construction est compatible au raffinement des décompositions
cellulaires et, par suite, aux restrictions;
elle fournit le morphisme de faisceaux annoncé.

Si la $(p,q)$-forme~$\omega$ est positive, 
il en est de même de~$f_*\omega$.

Si $f$ est injective, alors $f^*(f_*\omega)=\omega$; nous dirons
que $f_*\omega$ est le \emph{prolongement par zéro} 
de la $(p,q)$-forme~$\omega$.
Dans ce cas, $\omega$ est positive  (resp. nulle)
si et seulement si $f_*\omega$ l'est. 

\section{Orientations d'un espace paralinéaire}

\subsection{}
Soit $U$ une variété topologique connexe non vide de dimension~$n$.
On note $\Or(U)$ l'ensemble de ses orientations.
Il est vide lorsque $U$ n'est pas orientable,
et c'est un $\{\pm1\}$-torseur sinon.

Soit $V$ un ouvert connexe non vide de~$U$.
Si $U$ est orientable,
l'application naturelle de $\Or(U)$ dans~$\Or(V)$ 
est un isomorphisme  de~$\{\pm1\}$-torseurs.

\subsection{} 
Soit $X$ un espace paralinéaire et soit $n$ un entier
tel que $\dim(X)\leq n$.

Soit $\mathscr C$ un G-recouvrement de~$X$ par des cellules.
On dit qu'une cellule $C$ de~$X$ est $\mathscr C$-fine
s'il existe~$C'\in\mathscr C$ telle que $C\subset C'$.

Considérons alors le diagramme de $\{\pm1\}$-ensembles 
dont les objets sont les $\Or(\mathring C)$, pour $C$ parcourant
l'ensemble des cellules $\mathscr C$-fines de dimension~$n$,
et les flèches sont les isomorphismes naturels 
$\Or(\mathring D)\to \Or(\mathring C)$ 
lorsque $C$ et $D$ sont des cellules $\mathscr C$-fines
telles que $C\subset D$.
Notons $\Or_{\mathscr C}(X)$ la limite de ce diagramme 
dans la catégorie des $\{\pm1\}$-ensembles.
Un élément de $\Or_{\mathscr C}(X)$ 
est appelé \emph{orientation de~$X$} subordonnée au G-recouvrement~$\mathscr C$.

Lorsque $\dim(X)<n$, le diagramme ci-dessus est vide et sa limite
est le singleton~$\{1\}$.

Soit $\mathscr D$ un G-recouvrement de~$X$ par des cellules
qui est plus fin que le G-recouvrement~$\mathscr C$.
Toute cellule~$\mathscr D$-fine est $\mathscr C$-fine;
toute orientation de~$X$ subordonnée à~$\mathscr C$
induit donc une orientation subordonnée à~$\mathscr D$.

L'application $\Or_{\mathscr C}(X)\to \Or_{\mathscr D}(X)$
ainsi définie est $\{\pm1\}$-équivariante
et injective. En effet, si $C$ est une cellule
de dimension~$n$ de~$X$, la famille $(C\cap D)_{D\in\mathscr D}$
est un G-recouvrement de~$C$; il existe donc une cellule $D\in\mathscr D$
telle que $C\cap D$ soit un espace paralinéaire de dimension~$n$;
il contient une  cellule de dimension~$n$.

On note $\Or_\groth(X)$ la limite inductive de cette
famille de $\{\pm1\}$-ensembles
$\Or_{\mathscr C}(X)$ lorsque $\mathscr C$ 
parcourt l'ensemble des G-recouvrements
de~$X$ par des cellules. Les applications canoniques
de $\Or_{\mathscr C}(X)$ dans~$\Or_\groth(X)$ sont injectives.

Lorsque $\dim(X)<n$, on a $\Or_\groth(X)=\{1\}$.

\subsection{}
Soit $f\colon Y\to X$ une immersion par morceaux d'espaces paralinéaires
de dimensions~$\leq n$.
   
Soit $\omega\in\Or_\groth(X)$ une orientation de~$X$.
On construit comme suit une orientation~$f^*\omega\in\Or_\groth(Y)$ de~$Y$.
Soit $\mathscr C$ un G-recouvrement de~$X$ par des cellules.
Soit $\mathscr D$ un G-recouvrement de~$Y$ par des cellules,
plus fin que le G-recouvrement $(f^{-1}(C))_{C\in\mathscr C}$,
et tel que $f$ induise un isomorphisme de~$D$ sur son image,
pour toute cellule $D\in\mathscr D$.

Soit $D$ une cellule de dimension~$n$ de~$Y$ qui est~$\mathscr D$-fine;
alors $f$ induit un homéomorphisme de~$\mathring D$ sur son image~$f(\mathring D)$ qui est l'intérieur d'une cellule~$\mathscr C$-fine~$C$ de dimension~$n$.
Notons $(f^*\omega)_D$ l'unique orientation de~$\mathring D$
qui se déduit de l'orientation~$\omega_{C}$ de~$\mathring C$
par~$f$.

La famille $f^*\omega$ ainsi définie est une orientation de~$Y$
(subordonnée au G-recouvrement~$\mathscr D$); 
elle ne dépend pas des choix faits dans sa construction.

En particulier, l'application $Y\mapsto \Or_\groth(Y)$ est un G-préfaisceau 
en $\{\pm1\}$-ensembles sur~$X$.

\begin{prop}
Le G-préfaisceau $Y\mapsto \Or_\groth(Y)$ 
est un faisceau en $\{\pm1\}$-ensembles.
\end{prop}
On l'appelle le G-faisceau des orientations sur~$X$
et on le note $\Or_{X_\groth}$.
\begin{proof}
Soit $(X_i)$ un G-recouvrement de~$X$ et, pour tout~$i$,
soit $\omega_i$ une orientation de~$X_i$; on suppose
que $\omega_i|_{X_i\cap X_j}=\omega_j|_{X_i\cap X_j}$
pour tout couple~$(i,j)$.
Choisissons pour tout~$i$ un G-recouvrement~$\mathscr C_i$
par des cellules auquel $\omega_i$ soit subordonné.
La réunion~$\mathscr C$ de ces familles est un G-recouvrement de~$X$.
Soit $C$ une cellule de dimension~$n$ de~$X$ qui est~$\mathscr C$-fine;
lorsque $i$ parcourt des indices tels que $C$ soit~$\mathscr C_i$-fine,
les orientations $(\omega_i)_C$ sont égales à une même orientation
de~$\mathring C$, que l'on note~$\omega_C$.
La famille~$\omega$ de~$\omega_C$ est une orientation de~$X$ subordonnée
au G-recouvrement~$\mathscr C$; elle vérifie $\omega|_{X_i}=\omega_i$
pour tout~$i$, et c'est la seule orientation de~$X$ qui possède
cette propriété.
\end{proof}

\begin{lemm}\label{lemm.M-Or}
Soit $X$ un espace paralinéaire de dimension~$\leq n$ 
et soit $\mathscr M$ un $\mathscr A_{X_\groth}$-module
annulé par~$\mathscr I_n$.  Soit $Y$ un sous-espace paralinéaire
de~$X$.

\begin{enumerate}
\item Pour toute orientation~$\eps\in\Or_G(Y)$,
l'application canonique
\[ \theta_\eps\colon \mathscr M (Y) 
 \to (\mathscr M \times^{\{\pm1\}} \Or_{X_\groth}) (Y),
\qquad m\mapsto [m,\eps], \]
est bijective.
\item Pour tout couple $(\eps,\eps')$ d'orientations de~$Y$,
il existe une unique section inversible 
$\alpha\in(\mathscr A_{X_\groth}/\mathscr I_n) (Y)$
telle que $\theta_\eps(\alpha m)=\theta_{\eps'}(m)$
pour tout~$m\in \mathscr M(Y)$.
Cette section~$\alpha$ est G-localement constante
et vérifie $\alpha^2=1$.
\end{enumerate}
\end{lemm}
\begin{proof}
\begin{enumerate}
\item
Soit $m_1,m_2\in\mathscr M(Y)$ tels que $\theta_\eps(m_1)=\theta_\eps(m_2)$.
Soit $\mathscr C$ un G-recouvrement de~$Y$ par des cellules
telle que que $(m_1,\eps)$ et $(m_2,\eps)$
induisent le même élément de $\mathscr M(C)\times^{\{\pm1\}}\Or_G(C)$
pour toute cellule~$C\in\mathscr C$.
On donc $m_1|_C=m_2|_C$ pour toute $n$-cellule~$C$ appartenant à~$\mathscr C$,
car l'action de~$\{\pm1\}$ sur~$\Or_G(C)$  est alors libre.
Par suite, la restriction à~$Y^{(n)}$ de~$m_1-m_2$ est nulle.
On a donc $m_1-m_2\in \mathscr I_n \mathscr M$
(lemme~\ref{lemm.noyau}), d'où $m_1=m_2$ puisque $\mathscr M$
est annulé par~$\mathscr I_n$.

Soit maintenant $\mu$ une section sur~$Y$ du G-faisceau 
 $\mathscr M \times^{\{\pm1\}} \Or_{X_\groth}$.
Il existe un G-recouvrement~$\mathscr C$ de~$Y$ par des cellules
et, pour toute cellule~$C$ de~$\mathscr C$,
un couple $(m_C,\eps_C)\in\mathscr M(C)\times^{\{\pm1\}}\Or_G(C)$ 
qui définit $\mu|_C$.  Quitte à raffiner ce G-recouvrement,
on suppose que l'orientation~$\eps$ de~$Y$ est~$\mathscr C$-fine
et que $\eps_C\in\Or(\mathring C)$ pour toute cellule~$C$
telle que $\dim(C)=n$. 
Quitte à changer $\eps_C$ en~$-\eps_C$ et~$m_C$ en~$-m_C$,
on suppose que $\eps_C=\eps|_C$ pour toute cellule~$C$.

Alors les restrictions à~$C\cap C'$
de~$[m_C,\eps|_C]$ et~$[m_{C'},\eps|_{C'}]$ 
coïncident avec $\mu|_{C\cap C'}$.
D'après l'injectivité, déjà établie, appliquée sur~$C\cap C'$,
on a $m_C|_{C\cap C'}=m_{C'}|_{C\cap C'}$.
Les sections~$m_C$ se recollent donc en une section~$m\in\mathscr M(Y)$;
par construction, on a $\theta_\eps(m)=\mu$.

\item
Il existe un G-recouvrement~$\mathscr C$ de~$Y$ par des cellules
tel que les orientations~$\eps$ et~$\eps'$ soient $\mathscr C$-fines.
Soit $C$ une cellule appartenant à~$\mathscr C$.
Si $\dim(C)=n$, alors $\eps|_C$ est l'une des deux orientations de~$\mathring C$,
ainsi que $\eps'|_C$; soit $\alpha_C$ l'unique élément de~$\{\pm1\}$
tel que $\eps'|_C= \alpha_C\cdot\eps_C$.
Sinon, on a $\dim(C)<n$ et  on pose $\alpha_C=1$.

Soit $C$ et~$C'$ des cellules appartenant à~$\mathscr C$.
Supposons que $\dim(C\cap C')=n$.
Alors, $C\cap C'$ contient une cellule~$D$ de dimension~$n$;
les applications $\Or(\mathring C)\to \Or(\mathring D)$
et $\Or(\mathring C')\to \Or(\mathring D)$ sont des isomorphismes
de $\{\pm1\}$-torseurs, donc $\alpha_C=\alpha_{C'}$.
Il existe alors une unique 
section~$\alpha\in(\mathscr A_{X_\groth}/\mathscr I_n)(Y)$
telle que $\alpha|_C=\alpha_C$ pour toute cellule~$C$ 
appartenant à~$\mathscr C$.
On a $\alpha^2=1$, donc $\alpha$ est inversible,
et $\theta_{\eps'}^{-1}=\alpha \theta_\eps^{-1}$.
\end{enumerate}
\end{proof}

\begin{coro}
Il existe une unique structure
de $\mathscr A_{X_\groth}/\mathscr I_n$-module
sur le produit contracté 
$ \mathscr M \times^{\{\pm1\}} \Or_{X_\groth}$
telle que pour tout sous-espace paralinéaire~$Y$ et toute
orientation~$\eps$ de~$Y$,
l'application~$\theta_\eps$ soit linéaire.
\end{coro}
\begin{proof}
Le lemme précédent entraîne immédiatement
l'existence et l'unicité d'une telle structure  de module
au-dessus de tout sous-espace paralinéaire~$Y$ de~$X$
qui possède une orientation.
L'assertion en découle, car $X$ est G-recouvert par de tels sous-espaces.
\end{proof}

\section{Formes-volume sur un espace paralinéaire}

\subsection{}
On appelle \emph{espace paralinéaire de niveau~$n$}
un espace paralinéaire de dimension~$\leq n$. Dans cette terminologie,
l'entier~$n$ est une \emph{donnée} attachée à l'espace paralinéaire
considéré et non une simple propriété  de cet espace.
En particulier,
les sous-espaces paralinéaires d'un espace paralinéaire de niveau~$n$
seront considérés comme espaces paralinéaires de niveau~$n$,
même si leur dimension peut être strictement plus petite.

La théorie que nous introduisons ci-dessous 
est essentiellement vide lorsque les espaces sont
de dimension~$<n$, mais nos constructions ultérieures requerront la souplesse
de ce langage, notamment lorsqu'interviennent des restrictions
à des parties paralinéaires.

\begin{defi}\label{defi.faisceau-formes-volume}
Soit $\mathscr G$ une classe de coefficients.
Soit $X$ un espace paralinéaire de niveau~$n$.
On appelle G-faisceau des formes-volume à coefficients dans~$\mathscr G$ 
sur~$X$ le produit contracté
\[ \abs{\mathscr G^{n,0}_{X_\groth}}  
= \mathscr G^{n,0}_{X_\groth}  \times^{\{\pm1\}} \Or_{X_\groth}. \]
\end{defi}

L'espace~$\R^n$ est muni de la forme-volume canonique~$\omega_\R^n$
associée à la forme \mbox{$\di x_1\wedge\dots\wedge\di x_n$}
et à l'orientation canonique.

\subsection{}
Le G-faisceau sur~$X$ des formes-volume à coefficients dans~$\mathscr G$
est un $\mathscr G_{X_\groth}/\mathscr I_n\mathscr G_{X_\groth}$-module 
G-localement libre de rang~$1$
(proposition~\ref{prop.A(n,q)-libre} et lemme~\ref{lemm.M-Or}).

Lorsque $X$ est paracompact, ces résultats prouvent en fait que
ce G-faisceau est libre de rang~$1$ et en fournissent 
une classe explicite de  bases. Rappelons leur description.

Soit $\mathscr C$ une décomposition cellulaire de~$X$
et, pour toute cellule $C\in\mathscr C$, soit $f_C\colon C\to\R^n$
un isomorphisme de~$C$ sur une cellule de~$\R^n$.
Pour tout $C\in\mathscr C_n$, on dispose de la $(n,0)$-forme 
$f_C^*(\di x_1\wedge\dots\wedge\di x_n)$ sur~$C$
et de l'orientation~$\eps_C$ de~$C$ déduite de celle de~$\R^n$
via~$f_C$; notons~$\omega_C$ la $n$-forme volume associée sur~$C$.
Lorsque $C\in\mathscr C$ est de dimension~$<n$, on pose $\omega_C=0$.
La famille $(\omega_C)_{C\in\mathscr C}$ définit une $n$-forme volume sur~$X$
qui est une base du G-faisceau des $n$-formes-volume.
On dira qu'une telle $n$-forme-volume est standard.\index{forme-volume sur un espace paralinéaire!standard}

Soit $\alpha$ une autre base standard, associée
à une décomposition cellulaire~$\mathscr D$ et à une famille $(g_C)$
d'isomorphismes. 
Soit $C\in\mathscr C$ et $D\in\mathscr D$ des cellules
telles que $\dim(C\cap D)=n$; on a $\alpha_D|_{C\cap D}=\abs{J_{g_D/f_C}}\omega_C$, où $J_{g_D/f_C}$ désigne le jacobien relatif de l'immersion de~$C\cap D$ 
donnée par~$g_D$
par rapport à celle donnée par~$f_C$.

Cela prouve que deux formes-volume standard diffèrent par multiplication
par un élement G-localement constant strictement positif
de~$\mathscr F_{X_\groth}/\mathscr I_n\mathscr F_{X_\groth}$.

\begin{defi}
Soit $\mathscr G$ une classe de coefficients.
Soit $X$ un espace paralinéaire de niveau~$n$ 
et soit $\omega$ une forme-volume sur~$X$.
On dit que $\omega$ est \emph{positive} 
si pour tout sous-espace paracompact~$Y$ de~$X$
et toute forme-volume standard~$\alpha$ sur~$Y$,
l'unique élément~$u$ de~$\mathscr G_{X_\groth}/\mathscr I_n\mathscr G_{X_\groth}(Y)$
tel que $\omega = u\alpha$ est positif.
\end{defi}

C'est une propriété G-locale.

\subsection{}\label{sss.forme-volume-f^*}
Soit $f\colon X\to Y$ un morphisme
d'espaces paralinéaires de niveau~$n$.
En revenant à la définition d'un morphisme paralinéaire
(\ref{defi.morphisme-pl}), 
on voit qu'il existe  un G-recouvrement $(X_i)$ de~$X$ et
un G-recouvrement $(Y_j)$ de~$Y$ par des cellules
vérifiant la propriété suivante:
pour tout~$i$, il existe~$j$ tel que $f(X_i)\subset Y_j$
et tel que, soit $\dim(f(X_i))<n$, soit $f$ est un isomorphisme
de~$X_i$ sur un sous-espace paralinéaire de~~$Y_j$.

Soit $\omega$ une forme-volume sur~$Y$ à coefficients dans
une classe de coefficients~$\mathscr G$.
Pour tout~$j$, soit $\eps_j$ une orientation de~$Y_j$
et soit $\omega_j$ la $n$-forme sur~$Y_j$ 
telle que $\omega|_{Y_j}=[\omega_j,\eps_j]$.
Il existe une unique forme-volume  à coefficients dans~$\mathscr G$
sur~$X$ que l'on note $f^*\omega$
dont la restriction à~$X_i$ est nulle si $\dim(X_i)<n$,
et vaut $[f^*\omega_j,f^*\eps_j]$ sinon.
Elle ne dépend pas des choix faits.

Si la forme-volume~$\omega$ est positive, 
il en est de même de~$f^*\omega$.

L'application $\omega\mapsto f^*\omega$ 
est un morphisme de faisceaux
\[ f^* \colon  \abs{\mathscr G^{n,0}_{Y_\groth}}
 \to f_* \abs{\mathscr G^{n,0}_{X_\groth}}. \]

\subsection{}\label{sss.forme-volume-f_*}
Soit $f\colon X\to Y$ une immersion par morceaux compacte
d'espaces paralinéaires de niveau~$n$.
Soit $\mathscr G$ une classe de coefficients.
Nous allons construire un morphisme de faisceaux
\[ f_*\colon  f_*\abs{\mathscr G^{n,0}_{X_\groth}}
 \to  \abs{\mathscr G^{n,0}_{Y_\groth}}. \]

Soit $U$ un sous-espace paralinéaire paracompact de~$Y$.
Alors, $f_U\colon f^{-1}(U)\to U$ est une immersion par morceaux
compacte et $f^{-1}(U)$ est paracompact.
D'après la proposition~\ref{prop.decomp-relative},
il existe un couple $(\mathscr C,\mathscr D)$ 
de décompositions cellulaires qui est adapté à~$f$.
Soit $\omega$ une forme-volume sur $f^{-1}(U)$.
Soit $D$ une cellule de~$\mathscr D$ telle que $\dim(D)=n$.
L'ensemble des cellules $C\in\mathscr C_n$ telles que $f(C)=D$
est fini; pour toute telle cellule~$C$, $f|_C$ induit un isomorphisme
de~$C$ sur~$D$, de sorte que $\omega|_C$ induit une forme-volume sur~$D$.
Notons $f_*\omega_D$ la somme de ces formes-volume.
Il existe une unique forme-volume $f_*\omega$ sur~$U$ 
dont la restriction à toute $n$-cellule~$D\in\mathscr D$
est égale à~$f_*\omega_D$.

Cette construction est compatible au raffinement des décompositions
cellulaires et, par suite, aux restrictions;
elle fournit le morphisme de faisceaux annoncé.

Si la forme-volume~$\omega$ est positive, il en est de même de~$f_*\omega$.

Si $f$ est injective, alors $f^*(f_*\omega)=\omega$; nous dirons
que $f_*\omega$ est le \emph{prolongement par zéro} de la forme-volume~$\omega$.
Notons que $\omega$ est positive  (resp. nulle)
si et seulement si $f_*\omega$ l'est. 

\subsection{}
Soit $X$ un espace paralinéaire et soit $n$ un entier.
L'ensemble des sous-espaces paralinéaires~$\Sigma$ de~$X$,
fermés et de dimension~$\leq n$, est un ensemble ordonné filtrant;
si $j\colon\Sigma\to \Tau$ est une inclusion, on dispose d'une application 
$j_*\colon \abs{\mathscr G^{n,0}}(\Sigma)\to \abs{\mathscr G^{n,0}}(\Tau)$.
Ces applications font de la famille $(\abs{\mathscr G^{n,0}}(\Sigma))$
un système inductif dont la limite est notée $\mathscr V^{\mathscr G,n}_\pf(X)$,
ou $\mathscr V^n_\pf(X)$ s'il n'y pas d'ambiguïté possible quant à la classe~$\mathscr G$.
Plus généralement,
cette construction définit un préfaisceau~$\mathscr V^{\mathscr G,n}_{X_\groth,\pf}$ sur~$X_\groth$.

Pour tout $\Sigma$, l'application canonique
de~$\abs{\mathscr G^{n,0}}(\Sigma)$ dans~$\mathscr V^n_\pf(X)$ est injective.

On notera $\mathscr V^{\mathscr G,n}_{X_\groth}$  (ou $\mathscr V^n_{X_\groth}$
s'il n'y a pas d'ambiguïté sur la classe~$\mathscr G$)
le G-faisceau associé sur~$X_\groth$
et on dira que c'est le \emph{faisceau des formes-volume $n$-dimensionnelles}
à coefficients dans~$\mathscr G$ sur~$X_\groth$.

Notons que si $\dim(X)\leq n$, 
alors $\mathscr V^n_\pf=\abs{\mathscr G^{n,0}_{X_\groth}}$,
donc $\mathscr V^n=\abs{\mathscr G^{n,0}_{X_\groth}}$.
 
\begin{lemm}
Soit $X$ un espace paralinéaire, soit $n$ un entier.
\begin{enumerate}
\item Soit $\Sigma$ et $\Sigma'$ des sous-espaces paralinéaires
de~$X$, fermés et de dimension~$\leq n$,
soit $\omega$ et~$\omega'$ des formes-volume $n$-dimensionnelles
sur~$\Sigma$ et~$\Sigma'$ respectivement, 
notons $j$ et~$j'$ les injections de~$\Sigma\cap \Sigma'$ dans~$\Sigma$
et~$\Sigma'$ respectivement.
Pour que les classes de~$\omega$ et~$\omega'$ soient égales 
dans~$\mathscr V^n_\pf (X)$, il faut et il suffit que les relations
\[ \omega|_{\Sigma\cap \Sigma'}=\omega'|_{\Sigma\cap\Sigma'},
\qquad 
\omega=j_*\omega|_{\Sigma\cap \Sigma'},
\quad\text{et}\quad
\omega'=j'_*\omega|_{\Sigma\cap \Sigma'} \]
soient satisfaites.
\item Le préfaisceau $\mathscr V^n_\pf$ sur~$X_\groth$ est séparé.
\item Si $X$ est paracompact, $\mathscr V^n_{\pf}$ est un faisceau.
\end{enumerate}
\end{lemm}
\begin{proof}
\begin{enumerate}
\item
Ces relations entraînent que les classes de~$\omega$ et~$\omega'$
sont toutes deux égales à celle de~$\omega|_{\Sigma\cap\Sigma'}$.
Supposons inversement que les classes de~$\omega$ et~$\omega'$
soient égales. Soit $\Tau$ un sous-espace paralinéaire
contenant~$\Sigma$ et~$\Sigma'$ tel que $i_*\omega=i'_*\omega'$,
où  $i$ et~$i'$ désignent les inclusions de~$\Sigma$ et~$\Sigma'$ dans~$\Tau$.
Pour prouver les relations demandées, on peut raisonner G-localement
et donc supposer que $\Tau$ possède une décomposition cellulaire
compatibles aux sous-espaces~$\Sigma$, $\Sigma'$ et~$\Sigma\cap\Sigma'$.
L'égalité se vérifie alors directement cellule par cellule et découle
directement de la définition des images directes des formes.

\item
Cela résulte immédiatement de
ce que la véracité des trois relations de~(1) est une propriété G-locale.

\item
Soit $(X_i)_{i\in I}$ un G-recouvrement de~$X$, pour tout~$i$, 
soit $\Sigma_i$ un sous-espace paralinéaire de~$X_i$, fermé dans~$X_i$
et de dimension~$\leq n$, et soit $\omega_i$ une forme-volume
$n$-dimensionnelle sur~$\Sigma_i$. Supposons que pour tout couple~$(i,j)$,
les images de~$\omega_i$ et~$\omega_j$
dans $\mathscr V^n_\pf(X_i\cap X_j)$ coïncident
et démontrons qu'il existe un sous-espace paralinéaire~$\Sigma$
de~$X$, fermé et de dimension~$\leq n$, et une forme-volume~$\omega$
sur~$\Sigma$ dont la classe dans~$\mathscr V^n(X_i)$ est celle de~$\omega_i$.

Comme le préfaisceau~$\mathscr V^n_\pf$ est séparé,
on peut raffiner le G-recouvrement~$(X_i)$ et, en particulier,
le supposer  localement fini.

Soit $(V_j)_{j\in J}$ un recouvrement localement fini
de~$X$ par des parties compactes dont les intérieurs recouvrent~$X$;
il en existe car $X$ est paracompact. C'est a fortiori un G-recouvrement.
Pour tout~$j$, il existe une partie \emph{finie}~$I_j$ de~$I$
telle que $\bigcup_{i\in I_j} X_i$ contient~$V_j$.
La famille $(X_i\cap V_j)$ indexée par l'ensemble des couples 
$(i,j)$ tels que $j\in J$ et $i\in I_j$ est un G-recouvrement localement fini
de~$X$, plus fin que le recouvrement~$(X_i)$ initial.

Soit $\Sigma$ la réunion des parties~$\Sigma_i$;
c'est une partie paralinéaire de~$X$, fermée et 
de dimension~$\leq n$;
de plus, 
la famille~$(\Sigma_i)$ est un G-recouvrement de~$\Sigma$.
D'après les relations de~(1), on a $\omega_i|_{\Sigma_i\cap\Sigma_j}=\omega_j|_{\Sigma_i\cap\Sigma_j}$, pour tout couple~$(i,j)$.
Il existe donc une $n$-forme-volume~$\omega$ sur~$\Sigma$
dont la restriction à~$\Sigma_i$ est égale à~$\omega_i$, pour tout~$i$.
Fixons~$i$ et démontrons que les classes de~$\omega|_{\Sigma_i}$
et de~$\omega_i$ coïncident dans~$\mathscr V^n_\pf(X_i)$.
Notons~$f$ l'injection de~$\Sigma_i$ dans~$\Sigma\cap X_i$; 
il suffit de démontrer l'égalité  $\omega|_{\Sigma\cap X_i}=f_*\omega_i$
de $n$-formes-volume sur~$\Sigma\cap X_i$.
Comme le préfaisceau~$\mathscr V^n_\pf$ est séparé, il suffit
de vérifier cette égalité après restriction à~$\Sigma_j\cap X_i$, pour tout~$j$.

Notons $\alpha$ la $n$-forme-volume $\omega|_{\Sigma_i\cap\Sigma_j}$.
Comme les classes de~$\omega_i$ et $\omega_j$ dans~$\mathscr V^n(X_i\cap X_j)$
coïncide, la forme-volume $\omega_i|_{\Sigma_i\cap X_j}$ est le prolongement
par zéro de~$\omega_i|_{\Sigma_i\cap\Sigma_j}=\alpha$,
et $\omega_j|_{\Sigma_j\cap X_i}$ est le prolongement par zéro
de $\alpha$.
Ainsi, $\omega|_{\Sigma_j\cap X_i}= \omega_j|_{\Sigma_j\cap X_i}$
est le prolongement par zéro de~$\omega_j|_{\Sigma_i\cap\Sigma_j}=\alpha$.
D'autre part, $(f_*\omega_i)|_{\Sigma_j\cap X_i}$
est le prolongement par zéro de $\omega_i|_{\Sigma_i\cap\Sigma_j}=\alpha$.
Cela démontre l'égalité requise et termine la démonstration du lemme.
\qedhere
\end{enumerate}
\end{proof}

\subsection{}
Soit $f\colon X\to Y$ une immersion par morceaux d'espaces linéaires
par morceaux. On a alors un morphisme d'image réciproque
$f^*\colon \mathscr V^n_{Y_\groth}\to f_*\mathscr V^n_{X_\groth}$,
déduit du morphisme défini dans le paragraphe~\ref{sss.forme-volume-f^*}
lorsque $\dim(X)\leq n$.
Si, de plus, $f$ est compacte, on a un morphisme  d'image directe 
$f_*\colon f_*\mathscr V^n_{X_\groth}\to\mathscr V^n_{Y_\groth}$,
déduit du morphisme défini dans le paragraphe~\ref{sss.forme-volume-f_*}
lorsque $\dim(X)\leq n$.
Dans les deux cas le point clé est que ces opérations 
commutent au prolongement par zéro.

\section{Mesure associée à une forme-volume}
\label{ss.mesure-forme-volume}

\subsection{}
Soit $X$ un espace paralinéaire de niveau~$n$.
Soit $\omega$ une forme-volume dimensionnelle à coefficients
boréliens sur~$X$. 
Le but de ce paragraphe est de lui associer une mesure (positive) de Borel 
régulière~$\omega_\Bor$ sur~$X$ et, si cette mesure est localement finie,
une mesure de Radon~$\omega_\Rad$ sur~$X$ dont $\omega_\Bor$ est la valeur absolue. 
Ces constructions définiront de plus des morphismes de faisceaux.

\subsection{}
Soit $X$ un espace paralinéaire de niveau~$n$ 
et soit $f\colon X\to \R^n$ une immersion d'espaces paralinéaires.
On note $\lambda_{X,f}$ l'unique mesure de Radon positive sur~$X$
telle que $f_*(\lambda_{X,f})$ soit la restriction à~$f(X)$
de la mesure de Lebesgue sur~$\R^n$.
En particulier, on a $\lambda_{X,f}=0$ si $\dim(X)<n$.

\subsection{}
Soit $X$ un espace paralinéaire de niveau~$n$.
Soit $\omega$ une forme-volume $n$-dimensionnelle sur~$X$
à coefficients boréliens.

a) Supposons d'abord qu'il existe une immersion linéaire
par morceaux $f\colon X\to\R^n$ qui induit un isomorphisme
de~$X$ sur un sous-espace paralinéaire de~$\R^n$, 
une fonction borélienne~$\alpha^\sharp$ sur un voisinage de~$f(X)$
telle que $\omega= f^*(\alpha^\sharp \omega_{\R^n})$.
Dans ce cas, on associe  alors à~$\omega$ la mesure 
de Borel $\omega_\Bor=f^*(\abs{\alpha^\sharp}) \lambda_{X,f}$ sur~$X$;
elle est positive et régulière.

Vérifions que cette mesure ne dépend pas des choix de~$f$ et~$\alpha^\sharp$.

Soit donc $g\colon X\to\R^n$ un isomorphisme de~$X$
sur un sous-espace paralinéaire de~$\R^n$ et soit
$\beta^\sharp$ une fonction lisse sur un voisinage de~$g(X)$
telle que $\omega= g^*(\beta^\sharp \omega_{\R^n})$.
Démontrons que l'on a
\[ f^*(\alpha^\sharp) \lambda_{X,f}=
 g^*(\beta^\sharp) \lambda_{X,g} .\]
On peut raisonner G-localement sur~$X$, et donc que l'application
$(f,g)\colon X\to\R^{2n}$ induit un isomorphisme de~$X$
sur une cellule.

Si $\dim(X)<n$, $\lambda_{X,f}=\lambda_{X,g}=0$
et les deux membres de l'égalité à démontrer sont nuls.

Supposons donc que $\dim(X)=n$.
Il existe alors un unique automorphisme affine de~$\R^n$, $u$,
tel que $g=u\circ f$.
Notons $J_u$ le jacobien de~$u$.
L'image directe par~$u$ de la mesure sur de Lebesgue sur~$\R^n$
est égale à la mesure de Lebesgue divisée par~$\abs {J_u}$.
On a $g_*(\lambda_{X,f})=u_*(f_*(\lambda_{X,f}))$,
donc $g_*(\lambda_{X,f})$ est la restriction à~$g(X)$
la mesure de Lebesgue divisée par~$\abs{J_u}$.
Par suite, $g_*(\lambda_{X,f})=\abs{J_u}^{-1} g_*(\lambda_{X,g})$,
d'où
\[ \lambda_{X,g} =\abs{J_u} \lambda_{X,f}. \]
D'autre part, on a 
$ u^*(\di x_1\wedge\dots\wedge \di x_n)= J_u (\di x_1\wedge\dots\wedge dx_n)$,
d'où 
\[ 
g^*(\di x_1\wedge\dots\wedge \di x_n)= J_u f^*(\di x_1\wedge\dots\wedge dx_n).
\]
De plus, si $\eps$ est l'orientation canonique de~$\R^n$,
l'orientation~$f^*\eps$ de~$X$ est égale à l'orientation~$g^*\eps$
multipliée par le signe de~$J_u$. On a donc
\[ g^* \omega_{\R^n}= \abs{J_u} f^*\omega_{\R^n} .\]
Il en résulte l'égalité 
\[ g^*(\beta^\sharp) = \abs{J_u}^{-1} f^*(\alpha^\sharp). \]
L'égalité cherchée en découle.

b) Considérons, de plus, un sous-espace paralinéaire~$Y$ de~$X$. 
Alors $f|_Y$ est une immersion
de~$Y$ dans~$\R^n$, $\alpha|_Y$ est une forme-volume sur~$Y$
et la mesure qui lui est associée est égale à la restriction
à~$Y$ de la mesure~$\alpha_\Bor$.

c) Les mesures de Borel régulières sur~$X$ formant
un faisceau pour la G-topologie (proposition~\ref{prop.mesures-B-faisceau},
on en déduit que 
lorsque $Y$ parcourt l'ensemble des sous-espaces paralinéaires
de~$X$ tels qu'il existe un couple $(f,\alpha^\sharp)$
comme ci-dessus,  les mesures $({\omega|_Y})_\Bor$ construites
se recollent en une mesure de Borel, positive et régulière,
$\omega_\Bor$ sur~$X$.

d) On vérifie enfin immédiatement que
l'application $\omega\mapsto \omega_\Bor$ est un morphisme
de faisceaux de~$\abs{\mathscr B^{n,0}_{X_\groth}}$ dans~$\mathscr M_{\Bor,X}$.

\subsection{}
On dit que la forme-volume~$\omega$ est \emph{localement intégrable}
si la mesure $\omega_\Bor$ est localement finie.
Notons $\abs{\mathscr B^{n,0}_{X_\groth,\Lunloc}}$
le sous-faisceau de $\abs{\mathscr B^{n,0}_{X_\groth}}$
dont les sections sont les formes-volume localement intégrables.
C'est un sous $\mathscr A_{X_\groth}$-module de 
$\abs{\mathscr B^{n,0}_{X_\groth}}$.

Supposons que $\omega$ soit localement intégrable
et construisons 
une mesure de Radon~$\omega_\Rad$ dont $\omega_\Bor$ est la valeur absolue.
Reprenons pour cela chaque étape de la construction précédente.

Supposons d'abord qu'il existe une immersion linéaire
par morceaux $f\colon X\to\R^n$ qui induit un isomorphisme
de~$X$ sur un sous-espace paralinéaire de~$\R^n$, 
une fonction borélienne~$\alpha^\sharp$ sur un voisinage de~$f(X)$
telle que $\omega= f^*(\alpha^\sharp \omega_{\R^n})$.
Dans ce cas, la formule
$\omega_\Rad=f^*(\alpha^\sharp) \lambda_{X,f}$ définit une mesure
de Radon sur~$X$ telle que $\abs{\omega_\Rad}=\omega_\Bor$.

Les arguments qui montraient que la mesure~$\omega_\Bor$
ne dépend pas des choix de~$f$ et~$\alpha^\sharp$
prouvent qu'il en est de même pour la mesure de Radon~$\omega_\Rad$.

Les arguments des étapes~\emph b)  et~\emph c)
prouvent alors que lorsque $Y$  
parcourt l'ensemble des sous-espaces paralinéaires
de~$X$ tels qu'il existe un couple $(f,\alpha^\sharp)$
comme ci-dessus,  les mesures $({\omega|_Y})_\Rad$ construites
se recollent en une mesure de Radon~$\omega_\Rad$ sur~$X$.
On a $\abs{\omega_\Rad}=\omega_\Bor$.

Le morphisme $\omega\mapsto \omega_\Rad$
du faisceau~$\abs{\mathscr B^{n,0}_{X_\groth,\Lunloc}}$ 
dans le faisceau~$\mathscr M_{\Rad,X}$
des mesures de Radon sur~$X$
est $\mathscr A_{X_\groth}$-linéaire.

Bien sûr, toute forme-volume à coefficients mesurables et localement
bornés, en particulier toute forme-volume lisse
est localement intégrable.

Si $\omega$ est une $n$-forme-volume positive localement intégrable,
la mesure de Radon associée est positive.

%
%
%
 
\begin{lemm}
Soit $f\colon Y\to X$ une immersion par morceaux d'espaces
paralinéaires de niveau~$n$.
Soit $\alpha$ une forme-volume à coefficients boréliens sur~$X$.
On a  l'égalité 
$ ({f^*\alpha})_\Bor = f ^*(\alpha_\Bor)$
de mesures de Borel régulières positives sur~$Y$.
Si $\alpha$ est localement intégrable, alors $f^*\alpha$
est localement intégrable et l'on a  
l'égalité 
$ ({f^*\alpha})_\Rad = f ^*(\alpha_\Rad)$
de mesures de Radon sur~$Y$.
\end{lemm}
Comme $f$ est une immersion par morceaux, l'image
inverse de la mesure de Borel régulière~$\alpha_\Bor$ par~$f$ est
définie G-localement (\cf \ref{app-image-inverse-mesure});
de même pour les mesures de Radon.

\begin{proof}
L'espace $Y$ est G-recouvert par des cellules~$C$
telles que $f$ induise un isomorphisme de~$C$ sur son image,
laquelle possède une immersion paralinéaire dans~$\R^n$.
L'assertion découle alors de la construction de ces mesures.
\end{proof}

\begin{lemm}\label{lemm.f_*-Bor-Rad}
Soit $f\colon Y\to X$ une immersion par morceaux compacte 
d'espaces paralinéaires de niveau~$n$. 
Soit $\alpha$ une forme-volume à coefficients boréliens sur~$Y$.
On a  une inégalité 
$ ({f_*\alpha})_\Bor \leq  f _*(\alpha_\Bor)$
de mesures de Borel régulières positives sur~$X$,
qui est une égalité si $f$ est injective.
Si $\alpha$ est localement intégrable, alors $f_*\alpha$
est localement intégrable et l'on a  
l'égalité 
$ ({f_*\alpha})_\Rad = f _*(\alpha_\Rad)$
de mesures de Radon sur~$X$.
\end{lemm}
\begin{proof}
En raisonnant localement, on se ramène au cas où $X$ est paracompact.
Soit $(\mathscr C,\mathscr D)$ un couple de décompositions cellulaires
qui est adapté à~$f$.
Après restriction à toute cellule~$D$ de~$\mathscr D$, 
les assertions découlent directement des définitions des mesures
et des images directes.
\end{proof}

\subsection{}
Soit $X$ un espace paralinéaire et soit $n$ un entier.
En vertu du lemme~\ref{lemm.f_*-Bor-Rad}, 
les constructions précédentes induisent un morphisme de faisceaux
$\lambda_\Bor\colon \mathscr V^n_{X_\groth}\to \mathscr M_{\Bor,X}$
et un morphisme $\mathscr A_{X_\groth}$-linéaire
$\lambda_\Rad$ du sous-faisceau
de~$\mathscr V^n_{X_\groth}$ des formes-volume $n$-dimensionnelles
à coefficients boréliens
dont l'image est une mesure localement finie dans le faisceau
$\mathscr M_{\Rad,X}$ des mesures de Radon sur~$X$.

La formation de $\lambda_\Bor$ commute aux immersions par morceaux
injectives; celle de $\lambda_\Rad$ commute aux immersions par morceaux
compactes. En particulier, les deux commutent au prolongement
par zéro.

\section{Calibrages et intégrales des $(n,n)$-formes}

\begin{defi}
Soit $\mathscr G$ une classe de coefficients.
Soit $X$ un espace paralinéaire de niveau~$n$.
Le G-faisceau dual
\[ \ShHom_{X_\groth}
  (\abs{\mathscr G^{n,0}_{X_\groth}},
   \mathscr G_{X_\groth}/\mathscr I_n\mathscr G_{X_\groth}) \]
du faisceau des formes-volume sur~$X$ 
est appelé faisceau des \emph{calibrages} sur~$X$
(à coefficients dans~$\mathscr G$).
\end{defi}

Lorsque 
l'espace~$X$ est  paracompact, le faisceau des formes-volume
et son dual sont libres de rang~$1$. 
Si $\omega$ est une forme-volume standard, 
il existe un unique calibrage~$\mu$ tel que $\mu(\omega)=1$;
ces calibrages sont appelés standard.
Par définition, un calibrage est standard 
s'il applique toute forme-volume standard définie sur un sous-espace
paralinéaire
sur une section G-localement constante 
de~$\mathscr G_{X_\groth}/\mathscr I_n\mathscr G_{X_\groth}$
et strictement positive.
il suffit de le vérifier G-localement pour \emph{une} forme standard.

En général,
le G-faisceau des calibrages sur~$X$ 
est G-localement libre de rang~$1$ 
sur~$\mathscr G_{X_\groth}/\mathscr I_n\mathscr G_{X_\groth}$.
On dira qu'un calibrage est G-localement constant (resp. positif, resp.~...)
s'il est G-localement multiple d'un calibrage standard 
par une section G-localement constante (resp. positive,
resp.~...)
de~$\mathscr G_{X_\groth}/\mathscr I_n\mathscr G_{X_\groth}$.

\begin{exem}\label{exem.calibrages-dec-cell}
Soit $X$ un espace paralinéaire de niveau~$n$,
paracompact.

Soit $\mathscr C$ une décomposition cellulaire de~$X$;
pour toute cellule~$C$ appartenant à~$\mathscr C$,
soit $f_C$ un isomorphisme de~$C$ sur une cellule d'un 
espace affine~$V_C$.
Pour toute cellule~$C\in\mathscr C_n$,
choisissons une orientation~$\eps_C$
de l'espace vectoriel~$\overrightarrow{\langle f_C(C)\rangle}$
et une
fonction $\mu_C$ de classe~$\mathscr G$
de~$f_C(C)$ dans l'espace~$ \abs{\Lambda^n\overrightarrow{\langle f_C(C)\rangle}}$
des vecteurs-volumes sur cet espace vectoriel;
soit alors $\mu'_C $ l'unique fonction de~$f_C(C)$ à 
valeurs dans 
$\Lambda^n \overrightarrow{\langle f_C(C)\rangle }$ tel que $\mu_C$
soit la classe de~$(\mu'_C,\eps_C)$.

Il existe alors un unique calibrage~$\mu$ de~$X$
tel que pour toute cellule~$C$ de dimension~$n$
et 
tout élément $\omega\in A^{n,0}(\overrightarrow{\langle f_C(C)\rangle })$,
$\mu(f_C^*[\omega,\eps_C]) $
soit la classe dans~$\mathscr G/\mathscr I_n\mathscr G(C)$
de la fonction
déduite de $(\omega,\mu'_C)$ par 
l'accouplement de dualité
entre $A^{n,0}( \overrightarrow{\langle f_C( C)\rangle })$
et $\Lambda^n \overrightarrow{\langle f_C(C)\rangle}$
et composition par~$f_C$. On résume cette propriété
en écrivant $\mu|_C=f_C^*\mu_C$.

Si les fonctions~$\mu_C$ sont constantes (resp. positives,
resp. ...), alors le calibrage~$\mu$
est G-localement constant (resp. positif, resp. ...).

Pour tout calibrage~$\mu$ sur~$X$,
il existe une décomposition cellulaire~$\mathscr C$,
une famille $(f_C)$ d'isomorphismes
et une famille $(\mu_C)$ qui définissent~$\mu$ comme ci-dessus.

Si le calibrage~$\mu$ est G-localement constant (resp. positif),
on peut en outre exiger que pour toute cellule~$C$ de dimension~$n$,
la fonction~$\mu_C$ soit constante (resp. positive) sur~$C$.

Si $X$ est muni d'une immersion par morceaux~$f\colon X\to V$
vers un espace affine, on peut aussi exiger que $f_C=f|_C$ pour toute
cellule~$C$.

En particulier, 
si $X$ est un sous-espace paralinéaire d'un espace affine~$V$,
il existe une décomposition cellulaire forte~$\mathscr C$
de~$X$ telle que pour toute cellule~$C\in\mathscr C$, le calibrage~$\mu|_C$
soit donné par une fonction $\mu_C\colon C \to \abs{\Lambda^n \overrightarrow{\langle C\rangle}}$.
On dira alors que la décomposition cellulaire~$\mathscr C$ 
est adaptée au calibrage~$\mu$.
\end{exem}

\subsection{}
Soit $f\colon X\to Y$ une immersion par morceaux 
d'espaces paralinéaires de niveau~$n$
et soit $\nu$ un calibrage de~$Y$.
Soit $P$ un sous-espace paralinéaire de~$X$
tel que $f$ induit un isomorphisme de~$P$ sur son image.
Le calibrage~$\nu|_{f(P)}$ induit alors un calibrage~$(f^*\nu)_P$ de~$P$.
Il existe un unique calibrage~$f^*\nu$ de~$X$ dont la restriction
à~$P$ est égal à~$(f^*\nu)_P$, pour toute cellule~$P$ comme ci-dessus.

Si le calibrage~$\nu$ est positif (resp. strictement positif),
il en est de même du calibrage~$f^*\nu$.

Si le calibrage~$\nu$ est G-localement constant,
il en est de même du calibrage~$f^*\nu$.

\subsection{}
Supposons $f\colon X\to Y$ une immersion par morceaux propre
d'espaces paralinéaires de niveau~$n$
et soit $\mu$ un calibrage de~$X$. 
Nous allons construire un calibrage~$f_*\mu$ sur~$Y$.

Supposons d'abord que $Y$ est paracompact. 
Il en est alors de même de~$X$ car $f$ est propre.
Soit $(\mathscr C,\mathscr D)$ un couple de décompositions cellulaires
adapté à l'immersion~$f$ (prop.~\ref{prop.decomp-relative}).
Soit $D$ une cellule de dimension~$n$ de~$\mathscr D$.
L'ensemble des cellules~$C\in\mathscr C$ telles que $f(C)=D$
est fini. Pour toute telle cellule~$C$, $f$ induit un isomorphisme
de~$C$ sur~$D$. Notons $\nu_D$ la somme des calibrages de~$D$
induit par les calibrages~$\mu_C$ pour toutes ces cellules~$C$.
La famille $(\nu_D)_{D\in\mathscr D}$ définit un calibrage~$\nu$
de~$Y$.

Soit $Y'$ un sous-espace paralinéaire de~$Y$,
soit $X'=f^{-1}(Y')$, soit $f'\colon X'\to Y'$ l'immersion
par morceaux déduite de~$f$ par restriction; elle est propre.
Soit $(\mathscr C',\mathscr D')$ un couple de décompositions
cellulaires adapté à $f'$ et soit $\nu'$ le  calibrage de~$Y'$
déduit de~$\mu|_{X'}$ par la construction précédente.

Démontrons que l'on a $\nu'=\nu|_{Y'}$.
Il suffit de vérifier que pour toute cellule~$\Delta$ de~$Y'$
on a $\nu'|_{\Delta}=\nu|_{\Delta}$. 
Les intersections $(\Delta\cap D\cap D')$ pour $D\in \mathscr D$ et 
$ D'\in\mathscr D'$
forment un G-recouvrement de~$\Delta$. 
Il suffit donc de démontrer
que l'assertion lorsqu'il existe des cellules $D\in\mathscr D$
et $D'\in\mathscr D'$ telles que  $\Delta\subset D\cap D'$.
Pour toute cellule $C \in\mathscr C$ tel que $f(C)=D$,
$f$ induit un isomorphisme de $f^{-1}(\Delta)\cap C$,
et l'ensemble $f^{-1}(\Delta)$ est la réunion de
ces sous-espaces $f^{-1}(\Delta)\cap C$;
de plus, $f^{-1}(\mathring \Delta)$ 
est la réunion disjointe des $f^{-1}(\mathring \Delta)\cap C$.
Par suite $\nu|_\Delta$ est la somme, pour ces cellules~$C$,
des calibrages de~$\Delta$
déduits de $\mu|_{f^{-1}(\Delta)\cap C}$ au moyen de~$f$.
Par ailleurs, $\nu'|_{\Delta}$ se décrit de façon analogue
en remplaçant~$D$ par~$D'$, $f$ par~$f'$,
et les cellules~$C$ par les cellules~$C'\in\mathscr C'$
telles que $f(C')=D'$.
Soit $x$ un point de~$\mathring \Delta$; il appartient
à $\mathring D$ et à $\mathring D'$. Tout point
de $f^{-1}(x)$ appartient à une unique cellule~$C\in\mathscr C$
telle que $f(C)=D$, et à une unique cellule~$C'\in\mathscr C'$
telle que $f(C')=D'$, et l'on a 
$f^{-1}(\Delta)\cap C=f^{-1}(\Delta)\cap C'$.
Les deux sommes définissant $\nu|_\Delta$ et $\nu'|_\Delta$
s'identifient alors terme à terme.

Revenons à la situation générale, où $Y$ n'est plus supposé paracompact.
Il résulte de ce qui précède qu'il existe un unique calibrage de~$Y$,
que l'on note~$f_*\mu$,  tel que pour  tout sous-espace
paralinéaire paracompact~$Y'$  et tout couple
$(\mathscr C',\mathscr D')$ de décompositions cellulaires
adapté à $f|_{Y'}\colon f^{-1}(Y')\to Y'$,
la restriction à~$Y'$ de~$f_*\mu$
soit égale au calibrage de~$Y'$
construit comme ci-dessus.

Si le calibrage~$\mu$ est positif (resp. strictement positif),
il en est de même du calibrage~$f_*\mu$.

Si le calibrage~$\mu$ est G-localement constant,
il en est de même du calibrage~$f_*\mu$.

\subsection{}\label{sss.forme-volume-calibrage-pl}
Soit $X$ un espace paralinéaire de niveau~$n$
muni d'un calibrage~$\mu$ à coefficients 
dans une classe de coefficients~$\mathscr G$.
Ce calibrage est un morphisme de~$\abs{\mathscr G^{n,0}_{X_\groth}}$
dans~$\mathscr G_{X_\groth}/\mathscr I_n\mathscr G_{X_\groth}$.
Par composition avec l'opérateur~$\mathrm J$, il induit un morphisme
$\mathscr G^{0,n}_{X_\groth}$
dans $(\mathscr G_{X_\groth}/\mathscr I_n\mathscr G_{X_\groth})\times^{\{\pm1\}}\Or_{X_\groth}$.
Par produit tensoriel avec~$\mathscr G^{n,0}_{X_\groth}$, 
et en le multipliant par~$(-1)^{n(n-1)/2}$,
on en déduit alors un homomorphisme
\[ \mathscr G^{n,n}_{X_\groth} \to \mathscr V^{n}_{X_\groth} ,
 \qquad \omega \mapsto \langle \omega,\mu\rangle \]
du faisceau des $(n,n)$-formes vers celui des $n$-formes-volume sur~$X$.

Lorsque $\omega$ est une $(n,n)$-forme positive
et le calibrage~$\mu$ est positif, la forme-volume $n$-dimensionnelle
$\langle\omega,\mu\rangle$ est positive.

\subsection{}
Soit $X$ un espace paralinéaire de niveau~$n$
muni d'un calibrage~$\mu$ à coefficients boréliens.
Soit $\omega$ une $(n,n)$-forme sur~$X$ à coefficients boréliens.
Compte tenu des constructions du~\S\ref{ss.mesure-forme-volume},
nous disposons donc d'une mesure de Borel régulière
$\langle\abs{\omega},\abs{\mu}\rangle=\langle\omega,\mu\rangle_\Bor$ 
sur~$X$.

Si cette mesure est localement finie on dira que $\omega$
est localement $\mu$-intégrable 
(c'est par exemple le cas si $\mu$ et $\omega$
sont à coefficients localement bornés). 
Dans ce cas, on dispose alors d'une
mesure de Radon $\langle\omega,\mu\rangle_\Rad$,
que l'on notera souvent $\langle\omega,\mu\rangle$ par abus,
et dont la valeur absolue est la mesure 
$\lambda_{\langle\abs{\omega},\abs{\mu}\rangle}$.
Cette mesure est positive
si le calibrage~$\mu$ et la forme~$\omega$ sont positifs.

Les $(n,n)$-formes localement $\mu$-intégrables forment un sous-faisceau
en $\mathscr B_{X_\groth}$-module de $\mathscr B^{n,n}_{X_\groth}$,
et l'application $\omega\mapsto \langle\omega,\mu\rangle$ est linéaire.

On dira enfin que le calibrage~$\mu$ est localement intégrable
si toute forme à coefficients localement bornés
est localement $\mu$-intégrable. 
Sur une cellule, cela revient à dire que le calibrage 
est produit d'un calibrage standard par une fonction localement intégrable.
Les calibrages G-localement constants sont localement intégrables.

\begin{exem}
Lorsque $X$ est un sous-espace affine de dimension~$n$
d'un espace affine~$V$, $\omega$ une forme lisse
sur~$X$ et~$\mu$ un vecteur-volume sur~$X$,
on retrouve la notion d'intégrale définie 
dans le paragraphe~\ref{sss.prelim-integration}
à partir des idées de Lagerberg.

D'une certaine manière, les constructions de ce chapitre visaient
simplement à « G-faisceautiser »
cette formule élémentaire de Lagerberg sur un espace affine.
\end{exem}

\begin{prop}\label{prop.formes-volume-pl-f*}
Soit $f\colon Y\to X$ une immersion par morceaux propre 
d'espaces paralinéaire de niveau~$n$.

\begin{enumerate}
\item
Pour toute $(n,n)$-forme $\omega$ sur~$X$ et tout calibrage~$\mu$ de~$Y$,
on a l'égalité
$f_* \langle f^*\omega,\mu\rangle=\langle \omega,f_*\mu\rangle$ 
de formes-volume $n$-dimensionnelles sur~$X$.

 
\item
Pour toute $(n,n)$-forme~$\omega$ sur~$Y$ et tout calibrage~$\mu$ de~$X$,
on a l'égalité
$ \langle f_*\omega,\mu\rangle = f_* \langle \omega,f^*\mu\rangle$
de formes-volume $n$-dimensionnelles sur~$X$.
\end{enumerate}
\end{prop}
\begin{proof}
En raisonnant localement, on se ramène au cas où $X$ est paracompact.
Il existe alors un couple $(\mathscr C,\mathscr D)$ de décompositions
cellulaires qui est adapté à~$f$. Les assertions indiquées
se vérifient alors, cellule par cellule,
à partir des constructions des intégrales et de
la définition du calibrage~$f_*\mu$ (resp. du calibrage~$f^*\mu$).
\end{proof}

\section{Calibrages et intégrales des $(n-1,n)$-formes}

\subsection{}\label{sss.forme-bord-construction}
Soit $\mathscr G$ une classe de coefficients.
Soit $X$ un espace paralinéaire de niveau~$n$, 
soit $\mu$ un calibrage sur~$X$ 
et soit $\omega$ une forme de type~\mbox{$(n-1,n)$}
à coefficients dans~$\mathscr G$.

Supposons que $X$ soit paracompact et considérons un
atlas cellulaire~$(\mathscr C,f)$ 
tel que pour toute cellule~$C$ de dimension~$n$ de~$X$, 
$f_C$ soit un isomorphisme de~$C$ sur une cellule forte de~$\R^n$,
le calibrage~$\mu|_C$ soit de la forme~$f_C^*\mu_C$,
où $\mu_C$ est une fonction
de~$f_C(C)$ dans~$\Lambda^n \R^n$,
identifié à $\abs{\Lambda^n\R^n}$ en munissant 
$\R^n$ de son orientation canonique (cf.~\S\ref{exem.calibrages-dec-cell}),
et $\omega|_C$ est de la forme~$f_C^*(\alpha_C)$,
où $\alpha_C\in\mathscr G^{n-1,n}_{f_C(C)}(f_C(C))$.

Soit $F$ une cellule de dimension~$n-1$ appartenant à~$\mathscr C$.
Soit $C\in\mathscr C$ une $n$-cellule adjacente à~$F$.
Par définition d'un atlas cellulaire, $f_C(F)$ est une cellule
forte de~$\R^n$; munissons-la de l'orientation~$\eps_F$ déduite de
celle de~$\R^n$ par la règle
de la normale sortante à~$f_C(C)$.
Par contraction par~$\mu_C$ 
et multiplication par~$(-1)^{n(n-1)/2}$, 
on déduit de~$\alpha_C$
une $(n-1,0)$-forme sur~$f_C(C)$;
on la restreint à~$f_C(F)$, on la transforme en forme-volume
par le choix de l'orientation~$\eps_F$ de~$f_C(F)$,
et on la rapatrie sur~$F$ au moyen de l'isomorphisme~$f_C$.
On obtient ainsi une forme-volume à coefficients dans~$\mathscr G$ sur~$F$,
que l'on note 
$\langle\omega,\mu\rangle|_{(F,C)}$.
Elle est inchangée
lorsqu'on compose~$f_C$ par un automorphisme affine de~$\R^n$.

Soit $\mathscr C_{n-1}$ la réunion des cellules de dimension~$\leq n-1$
de~$\mathscr C$. C'est un sous-espace paralinéaire
de~$X$, de dimension~$\leq n-1$.
Il existe une unique forme-volume sur cet espace dont la restriction
à une cellule~$F$ de dimension~$n-1$
est la somme
\[ \sum_{\substack{C\supset F}} \langle\omega,\mu\rangle_{(F,C)}, \]
où $C$ parcourt l'ensemble des $n$-cellules de~$\mathscr C$ adjacentes à~$F$.
On note $\langle\omega,\mu\rangle_{(\mathscr C,f)}$
l'élément de~$\mathscr V^{n-1}(X)$ défini par cette forme-volume.
Il est à coefficients dans~$\mathscr G$.


\begin{lemm}\label{lemm.restriction-n-1}
Soit $U$ un ouvert de~$X$, soit $(\mathscr D,g)$ un
atlas cellulaire tel que les hypothèses de la construction précédente
soient satisfaites pour $\omega|_U$ et le calibrage~$\mu|_U$.
Alors, la restriction à~$U$ de la  forme-volume
$\langle\omega,\mu\rangle_{(\mathscr C,f)}$
est égale à la forme-volume
$\langle\omega|_U,\mu|_U\rangle_{(\mathscr D,g)}$. 
%
\end{lemm}
\begin{proof}
Pour simplifier les notations, on 
note $\alpha $ la $(n-1)$-forme-volume sur~$\mathscr C_{n-1}$
construite ci-dessus
dont la classe est égale à $\langle\omega,{\mu}\rangle_{(\mathscr C,f)}$ 
et $\beta$ la $(n-1)$-forme-volume sur~$\mathscr D_{n-1}$
dont la classe est $\langle{\omega|_U},{\mu_U}\rangle_{(\mathscr D,g)}$.

Commençons par traiter le cas où la décomposition cellulaire~$\mathscr D$
raffine~$\mathscr C$ et où,
pour toute cellule~$D$ de dimension~$n$ de~$\mathscr D$, 
on a $g_D=f_C|_D$,
où $C$ est l'unique cellule de~$\mathscr C$ qui contient~$D$.

Sous ces hypothèses, on a $\mathscr C_{n-1}\cap U\subset\mathscr D_{n-1}$;
démontrons alors que $\beta$ est le prolongement par zéro de~$\alpha|_U$.
Soit~$G$ une cellule de~$\mathscr D$, de dimension~$n-1$, et soit~$F$
l'unique cellule de~$\mathscr C$ telle que $\mathring G\subset\mathring F$.

 
Supposons d'abord que $\dim(F)=n$.
Comme $U$ est ouvert dans~$X$, la cellule~$G$ 
est adjacente  à exactement deux cellules de dimension~$n$
de~$\mathscr D$, disons~$D'$ et~$D''$, toutes deux contenues dans~$F$.
Dans la définition de~$\beta$,
les deux termes $\beta_{(G,D')}$ et $\beta_{(G,D'')}$ sont opposés,
car ils sont donnés par la même~$(n-1,0)$-forme sur~$G$ et les
deux orientations opposées de~$F$. Par suite, $\beta|_G=0$.
D'autre part, le prolongement par zéro de~$\alpha|_U$
est nul sur~$G$ car $\mathring G$ est disjoint de~$\mathscr C_{n-1}$.

Traitons maintenant le cas où $\dim(F)=n-1$.
Toute $n$-cellule~$D\in\mathscr D$ qui est 
adjacente à~$G$ est contenue dans une unique
$n$-cellule~$C_D$ de~$\mathscr C$ qui est adjacente à~$F$. 
Comme $U$ est ouvert dans~$X$, on obtient ainsi
une bijection de l'ensemble des $n$-cellules de~$\mathscr D$ adjacentes à~$G$
sur l'ensemble des $n$-cellules de~$\mathscr C$ adjacentes à~$F$.
Par construction, on a donc $\alpha|_G=\beta|_G$. 

Ces relations prouvent que $\beta$ est le prolongement par zéro de~$\alpha|_{U\cap\mathscr C_{n-1}}$. En particulier, les formes-volume $(n-1)$-dimensionnelles
sur~$U$,
$\langle\omega,\mu\rangle_{(\mathscr C,f)}|_U$
et
$\langle\omega|_U,\mu|_U\rangle_{(\mathscr D,g)}$, sont égales.

Traitons maintenant le cas général.
Soit $(\mathscr D',g')$ un atlas cellulaire adapté aux
atlas $(\mathscr C,f)$ et $(\mathscr D,g)$.
Soit $D'$ une cellule de dimension~$n$ de~$\mathscr D'$,
soit $D\in\mathscr D$ et $C\in\mathscr C$ les uniques cellules
contenant~$D'$. Par construction,
$g'_{D'}$ se déduit de~$g_D|_{D'}$
et de~$f_C|_{D'}$ par composition par des automorphismes affines de~$\R^n$.
D'après le cas déjà traité, 
on a
\[ 
\langle\omega,\mu\rangle_{(\mathscr C,f)}|_U
= 
\langle\omega,\mu\rangle_{(\mathscr D',g')}
= 
\langle\omega,\mu\rangle_{(\mathscr D,g)},
\]
d'où le lemme.
\end{proof}

\subsection{}\label{sss.forme-bord-calibrage-pl}
Il découle en particulier du lemme~\ref{lemm.restriction-n-1}
que les formes-volume ainsi construites ne dépendent pas
du choix du couple~$(\mathscr C,f)$,
et que leur formation commute à la restriction à un sous-espace ouvert.

Il existe alors une unique forme-volume $(n-1)$-dimensionnelle sur~$X$
dont la restriction à tout \emph{ouvert} paracompact~$U$
est donnée par la construction précédente.
On la note $\langle\omega,{\mu}\rangle$; elle est à coefficients
dans~$\mathscr G$.

L'application $\omega\mapsto \langle\omega,\mu\rangle$ 
se prolonge en un morphisme de faisceaux $\mathscr G_{X_\groth}$-linéaire 
de~$\mathscr G^{n-1,n}_{X}$ dans~$\mathscr V^{n-1}_{X}$.

\subsection{}
En particulier, on dispose pour toute $(n-1,n)$-forme~$\omega$
à coefficients boréliens
d'une mesure de Borel régulière positive $\langle\omega,\mu\rangle_\Bor$.
On dit que $\omega$ est \emph{localement $\mu$-intégrable}
lorsque cette mesure est localement finie; on dispose alors 
d'une mesure de Radon  $\langle\omega,\mu\rangle_\Rad$.
Bien entendu, lorsque le calibrage est à coefficients localement
bornés, toute forme~$\omega$ à coefficients 
localement bornés, en particulier toute G-forme lisse,
est localement $\mu$-intégrable.\index{forme localement intégrable sur un espace paralinéaire calibré}

\begin{prop}\label{prop.formes-volume-bord-pl-f*}
Soit $f\colon Y\to X$ une immersion par morceaux propre 
d'espaces paralinéaires de niveau~$n$.

\begin{enumerate}
\item
Pour toute $(n-1,n)$-forme $\omega$ sur~$X$ et tout calibrage~$\mu$ de~$Y$,
on a l'égalité
$f_* \langle f^*\omega,\mu\rangle=\langle \omega,f_*\mu\rangle$ 
de formes-volume $(n-1)$-dimensionnelles sur~$X$.

 
\item
Pour toute $(n-1,n)$-forme~$\omega$ sur~$Y$ et tout calibrage~$\mu$ de~$X$,
on a l'égalité
$ \langle f_*\omega,\mu\rangle = f_* \langle \omega,f^*\mu\rangle$
de formes-volume $(n-1)$-dimensionnelles sur~$X$.
\end{enumerate}
\end{prop}
\begin{proof}
En raisonnant localement, on se ramène au cas où $X$ est paracompact.
Il existe alors un couple $(\mathscr C,\mathscr D)$ de décompositions
cellulaires qui est adapté à~$f$. Les assertions indiquées
se vérifient alors, cellule par cellule,
à partir des constructions des intégrales et de
la définition du calibrage~$f_*\mu$ (resp. du calibrage $f^*\mu$).
\end{proof}

\begin{prop}[Compensation des bords]
\label{prop.pl-compensation}
Soit $X$ un espace paralinéaire de niveau~$n$,
soit $\mu$ un calibrage sur~$X$
et soit $\omega$ une $(n-1,n)$-forme sur~$X$.
Soit $(X_i)_{i\in I}$ un G-recouvrement localement fini de~$X$;
pour toute partie~$J$ de~$I$, posons $X_J=\bigcap_{i\in J}X_i$
et notons $u_J$ l'immersion de~$X_J$ dans~$X$.
On a l'égalité de formes-volume $(n-1)$-dimensionnelles 
\[ \langle\omega,\mu\rangle = \sum_{J\neq\emptyset} (-1)^{\Card(J)-1}
   (u_J)_* \langle \omega|_{X_J},\mu|_{X_J}\rangle. \]
\end{prop}
\begin{proof}
D'après le lemme~\ref{lemm.restriction-n-1},
l'égalité à démontrer est locale. On peut donc supposer
que $X$ est paracompact, que les sous-espaces~$X_i$ sont fermés dans~$X$
et que l'ensemble~$I$ est fini.

Soit $(\mathscr C,f)$ un atlas cellulaire de~$X$
qui est adaptée au calibrage~$\mu$, à la forme~$\omega$ et
au recouvrement~$(X_i)$; on suppose aussi que pour 
pour toute $n$-cellule $C\in\mathscr C$, 
$f_C$ est un isomorphisme de~$C$ sur une cellule forte de~$\R^n$
et le calibrage~$\mu|_C$ est de la forme~$f_C^*\mu_C$,
où 
$\mu_C$ est une fonction de~$f_C(C)$ dans~$\Lambda^n\R^n$,
identifié à~$\abs{\Lambda^n\R^n}$ en munissant $\R^n$
de son orientation canonique;
on note aussi 
$\alpha_C\in\mathscr G^{n-1,n}_{f_C(C)}(f_C(C))$ l'unique
$(n-1,n)$-forme telle que
$\omega|_C=f_C^*(\alpha_C)$.

Pour chaque cellule~$F$ de~$\mathscr C$ de dimension~$n-1$
et chaque $n$-cellule~$C$ adjacente à~$F$,
on a défini plus haut une forme-volume $(n-1)$-dimensionnelle
$\langle \omega,\mu\rangle|_{(F,C)}$ sur~$F$
que l'on considère comme une forme-volume sur~$X$ portée par~$F$.
Par construction, 
la forme-volume $(n-1)$-dimensionnelle $\langle\omega,\mu\rangle$
sur~$X$
est donnée par 
\[\langle\omega,\mu\rangle=
 \sum_{F\supset C} \langle\omega,\mu\rangle_{(F,C)}. \]
Plus généralement, pour toute partie~$J$ de~$I$,
l'ensemble des cellules~$C\in\mathscr C$ qui sont contenues dans~$X_J$
est une décomposition celullaire de~$X_J$
et l'on a de même la relation
\[\langle\omega|_{X_J},\mu|_{X_J}\rangle=
 \sum_{F\subset C\subset X_J} \langle\omega,\mu\rangle_{(F,C)}. \]

Soit $F$ une cellule de dimension~$n-1$ de~$\mathscr C$ et $C$
une $n$-cellule adjacente à~$F$.
Dans le membre de gauche de la formule à démontrer,
le terme $\langle\omega,\mu\rangle_{(F,C)}$ apparaît une seule fois,
avec coefficient~$1$.
Analysons le membre de droite. Notons $I_C$ l'ensemble des~$i\in I$
tels que $C\subset X_i$; on a donc $C\subset X_J$ 
si et seulement si $J\subset I_C$. Comme les $X_i$
recouvrent~$X$, l'ensemble~$I_C$ n'est pas vide.
Ainsi, le terme
$\langle\omega,\mu\rangle_{(F,C)}$ 
apparaît avec le coefficient
\[ \sum_{\emptyset\neq J\subset I_C} (-1)^{\Card(J)-1}
  = 1 - (1-1)^{\Card(I_C)}  = 1. \]
Cela conclut la preuve de la proposition.
\end{proof}

\subsection{Discordance}
Soit $X$ un espace paralinéaire de niveau~$n$,
soit $\mu$ un calibrage sur~$X$
et soit $f\colon X\to E$ une application paralinéaire.

Soit $\mathscr C$ une décomposition cellulaire de~$X$
qui est adaptée au calibrage~$\mu$ et à l'application paralinéaire~$f$:
pour toute cellule~$C$ de~$\mathscr C$,
la restriction  $\mu|_C$ à~$C$ 
du calibrage~$\mu$
est donné par une fonction 
$\mu_C\colon C\to\abs{\Lambda^n\langle \vec C\rangle}$
et la restriction de~$f$ à~$C$ est affine.

Soit $F$ une cellule de dimension~$n-1$ de~$\mathscr C$
et soit~$\mathscr C_{>F}$ l'ensemble des 
cellules $C\in \mathscr C_n$ dont $F$ est une face. 
Fixons une orientation~$o$ sur le sous-espace affine~$\langle F\rangle$. 
Pour tout $C\in \mathscr C_{>F}$, munissons l'espace
affine~$\langle C\rangle$ de l'orientation~$o_C$ donnée par $o$ 
et la normale sortante, comme au paragraphe~\ref{sss.integrale.bord}. 
Pour tout $C\in \mathscr C_{>F}$, notons $\tilde\mu_C$ la
fonction de~$C$ dans~$\Lambda^n\langle\vec C\rangle$
telle que $\mu_C$ soit la classe du couple~$(\tilde\mu_C,o_C)$. 

L'application~$f|_C$ est affine, donc induit une application  linéaire
$\Lambda^n f|_C$ de~$\Lambda^n \vec{\langle C\rangle}$  dans~$\Lambda^n \vec E$.
Notons $\tilde\mu_F$
la somme des images des~$\Lambda^n f|_C(\tilde\mu_C|_F)$
pour~$C$ parcourant~$ \mathscr C_{>F}$; 
c'est une fonction de~$F$ dans~$\Lambda^n\vec E$.
La classe~$\mu_F$ du couple $(\tilde\mu_F,o)$ est une fonction
de~$F$ dans
le produit contracté $\Lambda^n\vec E\times^ { \{\pm1\} } \Or(F)$;
elle ne dépend pas du choix de l'orientation~$o$;
on dit que c'est la \emph{discordance}\index{discordance!d'un 
calibrage} du calibrage~$\mu$ 
le long de la face orientée~$F$, 
relativement à l'application~$f$.

\subsection{}
Considérons une décomposition cellulaire~$\mathscr C'$ 
qui raffine la décomposition cellulaire~$\mathscr C$
et étudions la discordance du calibrage~$\mu$
le long des cellules de dimension~$n-1$
appartenant à~$\mathscr C'$.

Soit $F'$  une telle cellule.

Supposons d'abord que $F'$ rencontre une $n$-cellule ouverte de~$\mathscr C$.
Alors, $F'$ 
est face de deux cellules
de dimension~$n$ de~$\mathscr C'$ portant le même vecteur-volume
et induisant des orientations opposées sur~$F'$.
Par suite, la discordance du calibrage~$\mu$ est nulle sur~$F'$.

Sinon, $F'$ est contenue dans une $(n-1)$-cellule~$F$
de~$\mathscr C$ et les deux discordances coïncident 
sur l'intérieur de~$F'$.

\subsection{}
Soit $X$ un espace paralinéaire de niveau~$n$, calibré, 
et soit $f\colon X\to E$ une application paralinéaire.

Si $X$ est paracompact, 
on dit que $X$ est \emph{harmonieux}\index{espace paralinéaire calibré harmonieux}
relativement à~$f$
s'il existe une décomposition cellulaire~$\mathscr C$ de~$X$, adaptée
à~$\mu$ et~$f$, telle que la discordance de~$\mu$ (relativement à~$f$)
le long de toute $(n-1)$-cellule appartenant à~$\mathscr C$ soit nulle.
Alors, tout ouvert de~$X$ est harmonieux.

Lorsque $X$ est harmonieux, le paragraphe précédent entraîne
que la propriété définissant l'harmonie
reste vraie pour toute décomposition
cellulaire de~$X$ qui est adaptée à~$\mu$ et~$f$.

Dans le cas général, on dit que $X$ est \emph{harmonieux}
si chacun de ses ouverts paracompacts est harmonieux.
Il suffit de vérifier cette propriété sur les membres
d'un recouvrement ouvert.

\begin{rema}\label{rema.harmonieux-calibrage-constant}
Soit $X$ un espace paralinéaire de niveau~$n$, calibré,
soit $f\colon X\to E$ une application paralinéaire.
Soit $C$ une $f$-cellule de dimension~$n$ telle que $\mathring C$
soit ouverte dans~$X$, et harmonieuse.
Il existe un unique calibrage $\mu\in \abs{\Lambda^n \langle f(C)\rangle}$
tel que le calibrage~$\mu_C$ de~$C$ soit constant, égal à~$f*\mu$.

Considérons en effet une décomposition cellulaire~$\mathscr C$
de~$C$ adaptée à son calibrage. Soit $D$ une $(n-1)$-cellule
appartenant à~$\mathscr C$; soit $C_1$ et~$C_2$ 
les deux $n$-cellules de~$\mathscr C$ dont le bord contient~$D$.
La discordance de~$\mu_C$ le long de~$D$ relativement à~$f$
est la différence des deux calibrages de~$C_1$ et~$C_2$
(considérés comme calibrages de~$\langle f(C)\rangle$.
Comme elle est nulle, les calibrages de~$C_1$ et~$C_2$ coïncident.
Comme le complémentaire dans~$C$ de la réunion des $(n-2)$-cellules
de~$\mathscr C$ est connexe, l'assertion en résulte.
\end{rema}

\begin{rema}
Dans le cas général, on peut regrouper les discordances telles
que définies ci-dessus en un objet global que nous allons
décrire succinctement mais qui n'interviendra pas dans la suite 
de ce travail.

Soit $\Sigma$ une partie paralinéaire de~$X$ de dimension~$\leq n-1$.
On va construire un morphisme de G-faisceaux $\{\pm1\}$-équivariant
du G-faisceau~$\Or_{\Sigma}$ des $(n-1)$-orientations de~$\Sigma$
dans le quotient du faisceau des fonctions boréliennes sur~$\Sigma$
à valeurs dans~$\Lambda^n \vec E$ modulo $(n-2)$-négligeabilité.
Cette construction est locale; cela permet de supposer que $\Sigma$
est une partie paralinéaire fermée d'un ouvert paracompact~$U$ de~$X$.
On choisit alors une décomposition cellulaire~$\mathscr C$ de~$U$
qui est adaptée à~$\mu$, à $f$ et à~$\Sigma$. On choisit
une orientation de~$\Sigma$; notons $\Sigma'$ la réunion des $(n-1)$-cellules
ouvertes de~$\mathscr C$ qui sont contenues dans~$\Sigma$.
Pour toute $(n-1)$-cellule~$F$ de~$\mathscr C$ qui est contenue dans~$\Sigma$,
on dispose alors du $n$-vecteur~$\tilde\mu_F\in\Lambda^n \vec E$;
on obtient une fonction de~$\Sigma'$ dans~$\Lambda^n \vec E$,
que l'on prolonge à tout~$\Sigma$ par~$0$.
À $(n-2)$-négligeabilité  près, cette fonction est indépendante du choix 
de la décomposition cellulaire.
Si l'on change l'orientation de~$\Sigma$ en l'orientation opposée,
elle fournit la fonction opposée, d'où l'équivariance annoncée.

Soit $\Tau$ une partie paralinéaire de~$X$ de dimension~$\leq n-1$
qui contient~$\Sigma$.
La restriction à~$\Sigma$ de la discordance de~$\Tau$ (relativement à~$f$)
est la discordance de~$\Sigma$ (relativement à~$f$).

L'espace~$X$ est harmonieux relativement à~$f$ si et seulement 
toutes ces fonctions discordance sont nulles.
\end{rema}

\subsection{}
Lorsque le calibrage~$\mu$ est G-localement constant, il existe un plus petit
sous-espace paralinéaire de~$X$, purement de dimension~$n-1$,
en dehors duquel $X$ est harmonieux.

L'assertion est locale, de sorte qu'on peut supposer que $X$ est paracompact.
On peut alors choisir la décomposition cellulaire~$\mathscr C$ 
de sorte
que pour toute cellule~$C$, la fonction~$\mu_C$ soit constante. 
Alors, les fonctions~$\mu_F$ sont constantes
et le support de la discordance (relative à~$\mathscr C$)
est la réunion des cellules~$F$ de dimension~$(n-1)$ telles
que $\mu_F\neq0$.  C'est bien un sous-espace linéaire par morceaux de~$X$
purement de dimension~$n-1$. 

\subsection{}
\label{sss.forme-bord-pl-discordance}
Soit maintenant $\omega\in\mathscr F^{n-1,n}_X(X)$ une G-forme 
de type~$(n-1,n)$ sur~$X$ qui est G-tropicalisée par
une application paralinéaire~$f\colon X\to E$,
c'est-à-dire que $\omega=f^*\alpha$,
où $\alpha$ est une forme de type~$(n-1,n)$
définie sur un voisinage paralinéaire de~$f(X)$ dans~$E$.
Expliquons alors comment calculer la forme-volume associée. 

Considérons comme précédemment une décomposition
cellulaire de~$X$ qui est adaptée à~$\mu$ et~$f$.
Soit $\partial\mathscr C$ la réunion
des cellules de~$\mathscr C$ dont la dimension
est au plus $n-1$; c'est un sous-espace paralinéaire
fermé de~$X$.

Avec les notations du~\S\ref{sss.forme-bord-construction},
Identifions~$E$ à un espace~$\R^N$
et écrivons $\alpha = \sum \alpha_{IJ} \di x_I \wedge \dc x_J$.
En identifiant formes-volume $(n-1)$-dimensionnelles sur~$F$
et $(n-1)$-formes grâce à l'orientation~$o$ de~$F$ choisie, on a 
\begin{align*}
\sum_{C\supset F} \langle\omega,\mu\rangle_{(F,C)} 
& =(-1)^{n(n-1)/2}
 \sum_{C\supset F} \sum_{I,J} f^*\alpha_{IJ}|_F f^*(\di x_I) \langle f^*(\dc x_J),\tilde\mu_C\rangle  \\
& = (-1)^{n(n-1)/2}
 \sum_{C\supset F} \sum_{I,J} f^*\alpha_{IJ}|_F f^*(\di x_I) \langle \dc x_J, \Lambda^n (\tilde\mu_C)\rangle  \\
& = (-1)^{n(n-1)/2}
 \sum_{I,J} f^*\alpha_{IJ}|_F f^*(\di x_I) \langle \dc x_J, \mu_F \rangle.
\end{align*}
En particulier, si la discordance de~$\mu$
le long de~$F$ est nulle  (relativement à~$f$),
la contribution de la cellule~$F$ 
à la mesure~$\langle\omega,\mu\rangle$ est nulle.

Lorsque $f$ est l'injection de~$X$ dans un espace affine~$\R^N$,
cette formule montre aussi que
la forme-volume $(n-1)$-dimensionnelle $\langle\omega,\mu\rangle$
se calcule sur~$F$
par contraction de la forme~$\omega$ et de
la discordance~$\mu_F$ du calibrage~$\mu$ le long de~$F$.

\subsection{}
Lorsque $X$ est un sous-espace paralinéaire d'un espace
affine~$E$ et que l'application~$f$ est l'injection canonique,
on parlera plus simplement de discordance et d'espace harmonieux,
en omettant la mention de~$f$.

\begin{prop}[Formules de Stokes et Green]
\label{prop.stokes-pl}
Soit $X$ un espace paralinéaire de niveau~$n$, calibré.

\begin{enumerate}
\item
Soit $\omega$ une $(n-1,n)$-forme de classe~$\mathscr C^1$ 
à support compact sur~$X$.
On a  la \emph{formule de Stokes}:
\begin{equation}
 \int_X \langle \di\omega,\mu\rangle =\int_X \langle\omega, \mu\rangle.
\end{equation}

\item
Soit $\alpha$ et $\beta$ des $(p,p)$ et $(q,q)$-formes 
\emph{symétriques} sur~$X$, où $p+q=n-1$,
de classe~$\mathscr C^2$.
Si l'intersection de leurs supports est compacte, 
on a la \emph{formule de Green}:
\begin{equation}
 \int_X \langle \alpha\wedge \ddc\beta - \ddc\alpha\wedge\beta, \mu\rangle
= \int_{X} \langle \alpha\wedge\dc\beta-\dc\alpha\wedge\beta , \mu\rangle.
\end{equation}
\end{enumerate}
\end{prop}
\begin{proof}
Comme le support de~$\omega$ est compact, il possède un voisinage
ouvert~$U$ qui est paracompact. Si $U$ est muni
du calibrage induit, on a $\int_X \langle \di\omega,\mu\rangle
=\int_U \langle \di\omega, \mu\rangle $
et $\int_X \langle \omega,\mu\rangle=\int_U \langle \omega,\mu\rangle$.
On est ainsi ramené au cas où $X$ est paracompact.

Soit alors $\mathscr C$ une décomposition cellulaire de~$X$
en cellules compactes adaptée au calibrage de~$X$. 
Par définition, $ \int_X \langle\di\omega,\mu\rangle$ 
est la somme (finie), sur toutes les cellules~$C$ de dimension~$n$ 
de~$\mathscr C$, de $\int_C \langle \di\omega,\mu\rangle$.
Démontrons que l'on a 
\[ \int_C \langle\di\omega,\mu\rangle = \int_C \langle \omega,\mu\rangle.\]
On identifie~$C$ à une cellule compacte d'un espace affine de dimension~$n$,
$V$. 
Soit $F$ une face de dimension~$n-1$ de~$C$; c'est l'intersection
de~$C$ avec un hyperplan affine~$H$ de~$V$; notons~$V_+$
le demi-espace de~$V$ qui contient~$C$.
Si le support de~$\omega$ ne rencontre aucune autre face
de dimension~$n-1$ que~$F$, la formule à démontrer 
s'identifie à la formule de Stokes
\[ \int_{V_+} \langle\di\omega,\mu\rangle = \int_{V_+} \langle \omega,\mu\rangle\]
du lemme~\ref{lemm.stokes-affine}.
Par un argument de partition de l'unité, on en déduit
le résultat voulu si le support de~$\omega$ ne rencontre
aucune face de dimension~$n-2$ de~$C$.

On raisonne alors par récurrence sur le nombre de faces
de dimension~$n-2$ qui rencontre le support de~$\omega$.
Soit $F$ une telle face
et soit $u\colon V\to\R^2$ une application affine surjective,
nulle sur~$C$.
Soit $\theta$ une fonction lisse, positive, à support compact sur~$\R^2$
qui vaut~$1$ au voisinage de~$0$.
Pour tout~$\eps>0$, on pose $\theta_{\eps}(x)=\theta(x/\eps)$.
Comme la forme $(1-\theta_\eps)\omega$ est nulle au voisinage de~$F$,
elle vérifie l'hypothèse de récurrence, c'est-à-dire
\[ \int_C \langle \di ((1-\theta_\eps)\omega),\mu\rangle = \int_C\langle (1-\theta_\eps)\omega,\mu\rangle. \]
Pour démontrer la proposition, il suffit de vérifier que lorsque $\eps\to0$,
\[ \int_C \langle \di(\theta_\eps\omega),\mu\rangle
 - \int_C \langle \theta_\eps\omega,\mu\rangle \]
tend vers~$0$.
Or, le volume dans~$C$ du support de $\theta_\eps\omega$ tend vers~$0$
en~$\mathrm O(\eps^{2})$, tandis que les coefficients
de $\di(\theta_\eps\omega)$ sont $\mathrm O(1/\eps)$ ;
le premier terme tend donc vers~$0$ en~$\mathrm O(\eps)$.
De même, pour toute face~$D$ de~$C$ de dimension~$n-1$,
le volume dans~$D$ du support de $\theta_\eps\omega$
tend vers~$0$ en~$\mathrm O(\eps)$, ce qui entraîne que
le second terme tend vers~$0$ en~$\mathrm O(\eps)$.

Comme dans la preuve du lemme~\ref{lemm.green-affine},
la formule de Green se déduit de la formule de Stokes
appliquée à la forme $\alpha \wedge \dc\beta-\dc\alpha\wedge\beta$.
\end{proof}

\chapter{Espaces G-tropicaux calibrés}

\section{Espaces G-tropicaux chordés}

\subsection{}
Soit $X$ un espace G-tropical.
Soit $\Sigma$ et $T$ des parties paralinéaires fermées de~$X$.
Si leur réunion $\Sigma\cup T$ est une partie paralinéaire de~$X$,
alors $\Sigma$ et $T$ en sont  des sous-espaces paralinéaires,
si bien que $\Sigma\cap T$ est une partie paralinéaire de~$X$.

En général, l'intersection 
de deux parties paralinéaires fermées~$\Sigma$ et~$T$ de~$X$ 
peut être très sauvage.
On obtient un tel exemple comme suit: $X$ est la réunion,
dans~$\R^2$ du graphe~$\Sigma$ de la fonction nulle 
et du graphe~$\Tau$ de la fonction $x\mapsto x \sin(1/x)$,
toutes deux définies sur~$[0;1]$. Il est muni de la structure
paralinéaire image réciproque de la structure standard de~$[0;1]$.
Les parties~$\Sigma$ et~$\Tau$ sont paralinéaires, fermées,
mais leur intersection n'est pas paralinéaire, 
en raison de son comportement au voisinage de l'origine.
 
Les espaces chordés que nous introduisons maintenant sont munis
d'une classe raisonnable de parties paralinéaires,
que nous appelons squelettiques, qui est en particulier stable
par intersection finie et par réunion finie lorsqu'elles sont fermées;
en outre, cette classe contiendra automatiquement
toutes les parties paralinéaires
qui sont purement de dimension « maximale » 
(proposition~\ref{prop.pur-n-squelette}).

\subsection{}\label{defi.squel-elem}
Un espace G-tropical \emph{chordé} est un espace G-tropical
muni d'un 
un ensemble de parties paralinéaires
de $X$ et de ses domaines,
appelées parties \emph{squelettiques élémentaires},
et soumises aux axiomes suivants : 
\begin{enumerate}
\item Si $V$ est un domaine de $X$, 
une partie squelettique élémentaire de $V$ est fermée dans~$V$ ;
\item Si $W$ est un domaine de~$V$ et si $\Sigma$
est une partie squelettique élémentaire de~$V$,
alors $\Sigma\cap W$ est une partie squelettique élémentaire de~$W$;
\item Si $(V_i)$ est une famille finie
de domaines fermés de $X$, si $\Sigma_i$ est pour tout $i$ une 
partie squelettique élémentaire de $V_i$
et si $\Sigma$ désigne la réunion 
des $\Sigma_i$, il existe une famille $(W_j)$
de domaines de $X$ qui G-recouvre $\Sigma$ et, pour tout $j$, 
une partie squelettique élémentaire de $W_j$ contenant $\Sigma\cap W_j$. 
%
\end{enumerate}

Tout domaine~$V$ d'un espace G-tropical chordé~$X$ 
hérite d'une telle structure: ses parties squelettiques
élémentaires  sont les parties squelettiques élémentaires 
de~$X$ qui sont contenues dans~$V$.

\subsection{}
Soit $X$ un espace G-tropical chordé.

On dit qu'une partie 
de $\Sigma$ de $X$ est \emph{squelettique}\index{partie squelettique d'un espace G-tropical chordé}
s'il existe une famille $(V_i)$ de domaines de $X$ 
tels que les $\Sigma\cap V_i$ forment un G-recouvrement de~$\Sigma$ 
et telle 
que pour tout $i$, 
l'intersection $\Sigma\cap V_i$ soit une partie paralinéaire fermée
d'une partie squelettique élémentaire de $V_i$.

Soit $\Sigma$ une partie de~$X$ et soit $V$ un domaine de~$X$
tel que $\Sigma\subset V$. Alors $\Sigma$ est squelettique dans~$X$
si et seulement si elle l'est dans~$V$.
(Si $\Sigma$ est squelettique dans~$V$,
la même famille $(V_i)$ convient; sinon, on la remplace
par la famille des~$V_i\cap V$ qui a même trace sur~$\Sigma$.)

Nous dirons qu'un point~$x$ de~$X$ est \emph{squelettique}\index{point squelettique}
s'il appartient à une partie squelettique de~$X$.
Cela revient au même de demander qu'il appartienne à une partie
squelettique élémentaire, ou encore que la partie~$\{x\}$ soit squelettique.

\begin{exem}\label{exem.paralineaire-chorde}
Soit $X$ un espace paralinéaire.
Les parties paralinéaires fermées des domaines de~$X$
le munissent d'une structure d'espace G-tropical chordé.
Ses parties squelettiques sont exactement ses parties paralinéaires.
\end{exem}

\begin{prop}\label{prop.squelettique-paralineaireX}
Soit $X$ un espace G-tropical chordé.

\begin{enumerate}
\item 
Toute partie squelettique de $X$ est paralinéaire.
\item
Soit $\Sigma$ une partie squelettique de~$X$
et soit $T$ une partie de~$\Sigma$. Pour que $T$ soit une partie
squelettique de~$X$, il faut et il suffit qu'elle soit paralinéaire.
\item
La réunion de deux parties squelettiques fermées de~$X$
est squelettique.
\item Toute partie de~$X$ qui est G-recouverte par des parties
squelettiques est squelettique.
\item
L'intersection de deux parties squelettiques de~$X$ est squelettique.
\end{enumerate} 
\end{prop}
\begin{proof}
\begin{enumerate}
\item Soit $\Sigma$ une partie squelettique de~$X$.
Par définition,
il existe une famille $(V_i)$ de domaines de $X$ 
tels que les $\Sigma\cap V_i$ constituent un 
G-recouvrement de~$\Sigma$ et telle 
que pour tout $i$, 
l'intersection $\Sigma\cap V_i$ soit une partie paralinéaire fermée
d'une partie squelettique élémentaire $S_i$ de $V_i$. En particulier, 
$\Sigma\cap V_i$ est paralinéaire. Il s'ensuit par le lemme \ref{lemme-partiespl-glocal}
que $\Sigma$ est paralinéaire. 

\item
Si $T$ est squelettique, elle est paralinéaire, d'après~a).
Supposons donc que $T$ est paralinéaire est démontrons
que c'est une partie squelettique.
Soit $(V_i)$ une famile de domaines de~$X$ 
telle que les $\Sigma\cap V_i$ forment un G-recouvrement de~$\Sigma$
et telle que, pour tout~$i$, 
$\Sigma\cap V_i$ soit une partie paralinéaire d'une partie squelettique 
élémentaire~$S_i$ de~$V_i$.
Pour tout $i$, $T\cap V_i$ est une partie paralinéaire de~$T$,
donc de~$S_i$, et les $(T\cap V_i)$ forment un G-recouvrement de~$T$.

\item 
Soient $\Sigma$ et $T$ deux parties
squelettiques fermées de~$X$. 
Soit $(V_i)$ une famille de domaines de~$X$
telle que les $\Sigma\cap V_i$
constituent un G-recouvrement de~$\Sigma$
et telle que pour tout~$i$, $\Sigma\cap V_i$
soit une partie paralinéaire d'une partie squelettique élémentaire~$\Sigma_i$
de~$V_i$. Choisissons de même une famille $(W_j)$ de domaines de~$X$
et, pour tout~$j$, une partie squelettique élémentaire~$T_j$ de~$W_j$,
qui attestent du caractère squelettique de~$T$.

Pour montrer que $\Sigma\cup T$ est squelettique, on peut par définition raisonner
G-localement sur $X$, ce qui permet de supposer que les $V_i$ et les $W_j$ sont fermés et en nombre fini. 
L'axiome c) de la définition d'une partie squelettique élémentaire assure alors l'existence 
d'une famille $(U_\ell)$ de domaines de $X$ qui G-recouvrent la réunion $R$ des $\Sigma_i$ et les $T_j$, 
et, pour tout $\ell$, d'une partie squelettique élémentaire $R_\ell$ de $U_\ell$ contenant
$R\cap U_\ell$. 
Mais alors les $U_\ell$ constituent un G-recouvrement de $\Sigma\cap T$ et pour tout $\ell$, la réunion 
$(\Sigma\cup T)\cap U_\ell$ est contenue dans $R_\ell$ ; et comme cette réunion est la réunion finie des $\Sigma
\cap V_i
\cap U_\ell$ et des $T\cap W_j\cap U_\ell$, c'est une partie paralinéaire de $R_\ell$. En conséquence, $\Sigma
\cup T$ est une partie squelettique
de $X$.

\item Soit $\Sigma$ une partie de $X$ qui est G-recouverte par une famille
$(\Sigma_i)$ de parties squelettiques. Pour montrer que $\Sigma$
est squelettique, on peut par définition raisonner localement pour la topologie
de $\Sigma$ induite par la G-topologie de $X$, 
ce qui permet de supposer que les $\Sigma_i$ sont des fermés de $X$ et que la
famille $(\Sigma_i)$ est finie, et l'on déduit alors de c)
que $\Sigma$ est squelettique. 

\item 
Soit $\Sigma, T$ des parties squelettiques de~$X$; démontrons
que $\Sigma\cap T$ est squelettique.
Soit $U$ un ouvert de~$X$ dans lequel $\Sigma$ est fermé,
soit $V$ un ouvert de~$X$ dans lequel $T$ est fermé.
Posons $W=U\cap V$. 
Puisque $\Sigma\cap W$ et $T\cap W$ sont fermés dans~$W$,
leur réunion est une partie squelettique de~$W$, d'après~\emph a),
dans laquelle chacune est paralinéaire. Par suite,
leur intersection $\Sigma\cap T\cap W=\Sigma\cap T$ est paralinéaire,
et donc squelettique d'après~\emph b). \qedhere
\end{enumerate} 
\end{proof}

\begin{defi}
Soit $X$ un espace G-tropical chordé.
On dit qu'une partie squelettique calibrée de~$X$ est pure de dimension~$p$
si son calibrage est G-localement constant et si sa restriction
à tout ouvert non vide est non nul.
\end{defi}

Une telle partie est purement de dimension~$p$. Par économie
d'écriture, nous dirons que son calibrage est « partout non nul ».

Soit $P$ une partie squelettique de dimension~$\leq p$ munie
d'un $p$-calibrage G-localement constant~$\mu_P$.
Le support de~$\mu_P$ est une partie squelettique fermée de~$P$,
purement de dimension~$p$ ; on le notera $\abs{(P,\mu_P)}$,
ou $\abs{P}$ si le calibrage~$\mu_P$ est sous-entendu.
C'est une partie squelettique calibrée pure de dimension~$p$.

\subsection{}\label{rema.polyedres-purs-G-faisceau}
Soit $P$ une partie squelettique calibrée fermée pure de dimension~$p$.
Si $V$ est un domaine de~$U$,
la partie squelettique $p$-calibrée $\abs{P\cap V}$ de~$V$ est pure.
Si $W$ est une partie paralinéaire de~$V$,
on a  $\abs{\abs{P\cap V}\cap W}=\abs{P \cap W}$:
ces deux parties squelettiques calibrées sont la réunion
des cellules de dimension~$p$ de~$P$ contenues dans~$W$
dont le calibrage est partout non nul.

On a ainsi muni la famille des parties squelettiques fermées $p$-calibrées
pures d'un domaine de~$X$ d'une structure de préfaisceau 
pour la G-topologie de~$X$.
\emph{Ce préfaisceau est un faisceau.}
En effet, soit $V$ un domaine de~$X$ et soit $(V_i)$
un G-recouvrement de~$V$; pour tout $i$, soit $P_i$ une partie
squelettique $p$-calibrée pure de~$V_i$; on suppose que pour tous~$i,j$,
on a $\abs{P_i\cap (V_i\cap V_j)}=\abs{P_j\cap (V_i\cap V_j)}$;
démontrons qu'il existe une unique partie squelettique $p$-calibrée pure~$P$ 
de~$V$
telle que $\abs{P\cap V_i}=P_i$ pour tout~$i$.

Établissons l'unicité. 
Considérons deux telles parties squelettiques~$P$ et~$Q$; 
on a $\abs{P\cap V_i}=\abs{Q\cap V_i}=P_i$ pour tout~$i$.
Soit $C$ une $p$-cellule compacte contenue dans~$P$
et démontrons qu'elle est contenue dans~$Q$, avec même calibrage.
Puisque $C$ est G-recouverte par les $V_i$, 
on peut supposer, quitte à considérer un G-recouvrement de~$C$,
qu'il existe~$i$ tel que $C\subset V_i$. Alors $C\subset P_i$
et son calibrage est celui hérité de~$P_i$.
Puisque $P_i=\abs{Q\cap V_i}$, elle est également contenue
dans~$Q$ et son calibrage est celui hérité de~$Q$. Cela conclut
la preuve de l'unicité de~$P$.

Prouvons maintenant l'existence de~$P$.
Soit $P$ la réunion des~$P_i$; démontrons que c'est une
partie squelettique fermée de~$V$. 
Il suffit de le vérifier localement; soit $x$ un point de~$V$;
comme les~$V_i$ forment un G-recouvrement de~$V$,
il existe un voisinage domanial compact~$W$ de~$x$ dans~$V$
et un ensemble fini~$I$ d'indices tels que $W\subset \bigcup_{i\in I} V_i$
et tels que pour tout~$i$, $V_i\cap W$ soit compact.

Démontrons que $\abs{P_j\cap W}\subset \bigcup_{i\in I} \abs{P_i\cap W}$
pour tout~$j$.

Soit $C$ une $p$-cellule contenue dans $\abs{P_j\cap W}$;
on a $C=\bigcup_{i\in I} C \cap V_i$, et chacune de ces intersections
est une partie paralinéaire compacte de~$C$;
par suite, $C$ est la réunion des $(C\cap V_i)^{(p)}$,
parties purement de dimension~$p$ de $C\cap V_i$.
Pour $i\in I$, on a 
\[ (C\cap V_i)^{(p)}\subset \abs{P_j\cap (V_i \cap V_j\cap W)}
=\abs{P_i\cap (V_i\cap V_j\cap W)} \subset \abs{P_i\cap W}. \]
Par suite, 
$ C \subset \bigcup_{i\in I} \abs{P_i\cap W}$ d'où,
\[ \bigcup_{i\in I} \abs{P_i\cap W} = \bigcup_j \abs{P_j\cap W}. \]

Prenant la trace sur l'intérieur~$\mathring W$ de~$W$ dans~$V$, 
on obtient
\[ \left(\bigcup_{i\in I} \abs{P_i\cap W}\right) \cap \mathring W
 = \left(\bigcup_j \abs{P_j\cap W}\right) \cap \mathring W
=\bigcup_j \left( \abs{P_j\cap W} \cap \mathring W\right)
= P \cap \mathring W \]
puisque 
pour tout~$j$, on a
\[  \abs{P_j\cap W}\cap \mathring W 
= \abs{\abs{P_j\cap W} \cap \mathring W}
= \abs{P_j\cap \mathring W} 
= P_j \cap \mathring W. \]
Cela prouve que 
\[ P\cap \mathring W = \bigcup_{i\in I} P_i \cap \mathring W. \]
Comme les~$P_i$ sont fermés, les parties $P_i\cap \mathring W$
sont des parties squelettiques fermées de~$\mathring W$;
leur réunion est alors une partie squelettique fermée de~$\mathring W$.
Cela conclut la démonstration que $P$ une partie squelettique fermée.

Démontrons que $P$ est G-recouverte par ses $p$-cellules
qui sont contenues dans l'un des~$P_j$.
Il suffit de le vérifier localement, après
restriction à un ouvert de la forme~$\mathring W$
comme ci-dessus, et cela résulte alors de la dernière formule établie.
Si $C$ est une $p$-cellule de~$P$
contenue dans~$P_i$ et dans~$P_j$, les calibrages de~$C$ déduits
des calibrages de~$P_i$ et de~$P_j$ coïncident,
par la condition de coïncidence des~$P_i$.
Ces calibrages munissent~$P$ d'un calibrage
qui, par construction, prolonge celui de~$P_i$.
En particulier, ce calibrage est partout non nul.

Démontrons que $\abs{P\cap V_i}=P_i$.
Puisque $P_i\subset P\cap V_i$, on en déduit
que $P_i \subset \abs{P\cap V_i}$.
Inversement, soit $C$ une $p$-cellule de $P\cap V_i$;
prouvons qu'elle est contenue dans~$P_i$.
Quitte à raisonner G-localement sur~$C$, on peut supposer
qu'il existe~$j$ tel que $C\subset P_j$.
Alors, $C\subset \abs{P_j\cap V_i}= \abs{P_i\cap V_j}\subset P_i$.

\subsection{}
Soit $P$ et $Q$ des parties squelettiques fermées $p$-calibrées pures de~$X$.

On dit que $P\preceq Q$ si $P\subset Q$ et si, pour toute $p$-cellule~$C$
de~$X$ contenue dans~$P$ sur lequel les calibrages de~$P$ et~$Q$
sont constants, le calibrage de~$C$ pour~$P$ est inférieur ou égal au 
calibrage de~$C$ pour~$Q$. Cette condition se vérifie G-localement.

La relation « $P\preceq Q$ » est une relation d'ordre dans l'ensemble
des parties squelettiques fermées $p$-calibrées pures de~$X$.

\begin{lemm}\label{lemm.squelette-calibre-treillis}
L'ensemble des parties squelettiques fermées $p$-calibrées pures de~$X$
est un treillis.
\end{lemm}
\begin{proof}
Soit $P$ et $Q$ deux parties squelettiques fermées $p$-calibrées pures
de~$X$. 
L'existence des bornes inférieure et supérieure de la paire~$\{P,Q\}$
se vérifie G-localement sur~$X$, et même G-localement sur 
la partie squelettique $P\cup Q$.

On définit ainsi G-localement deux parties squelettiques calibrées
de dimension~$p$, $P\wedge Q$ et $P\vee Q$, contenues dans $P\cup Q$.
La partie $P\cup Q$ est squelettique, et est G-recouverte
par les cellules compactes~$C$ telles que:
d'une part $C\subset P$ et le calibrage de~$P$ est constant sur~$C$
ou $\mathring C\cap P=\emptyset$, 
et d'autre part $C\subset Q$ et le calibrage de~$Q$ est constant sur~$C$
ou $\mathring C\cap Q=\emptyset$.
Il suffit de déterminer le calibrage d'une telle cellule
dans~$P\wedge Q$ et dans~$P\vee Q$.
Soit $C$ une telle cellule. Si $C\subset P$, on note~$\mu^P_C$
le vecteur-volume de~$C$ dans~$P$; sinon, on pose $\mu^P_C=0$.
On définit~$\mu^Q_C$ de façon similaire.
Les vecteurs-volume d'une cellule définissent une droite réelle orientée;
on définit alors $\mu^{P\wedge Q}_C=\inf(\mu^P_C,\mu^Q_C)$
et $\mu^{P\vee Q}=\sup (\mu^P_C,\mu^Q_C)$.
Les vecteurs-volume $\mu^{P\wedge C}_C$ 
définissent un calibrage~$\mu^{P\wedge Q}$ de $P\cup Q$;
on définit $P\wedge Q$ comme la partie squelettique $p$-calibrée fermée pure
$\abs{(P\cup Q, \mu^{P\wedge Q})}$; il est immédiat
que c'est la borne supérieure de~$\{P,Q\}$.
On définit de même $P\vee Q$. Ce sont pures.
\end{proof}

\section{Espaces G-tropicaux vertébrés}

\begin{defi}\label{defi.polyedre-variation-maximale}
Soit $X$ un espace G-tropical tel que $d_\trop(X)\leq n$ et
soit $(V,f\colon V\to\R^m)$ une carte G-tropicale de~$X$. 
On appelle \emph{polyèdre de variation maximale} de~$f$
l'ensemble $\Sigma_f^{(n)}$ des points de~$V$
en lesquels la dimension tropicale de~$f$ vaut~$n$.\index{polyèdre de variation maximale}
\end{defi}

C'est une partie fermée de~$V$, 
car l'ensemble des points
en lesquels la dimension tropicale de $f$ est strictement
inférieure à $n$ est ouvert par définition.

Dans cette notation, l'exposant~$n$ rappelle la borne
fixée sur la dimension tropicale de~$X$;
sa présence annonce la situation des espaces
analytiques où $\Sigma_f^{(n)}$ coïncidera
avec le lieu purement $n$-dimensionnel d'une partie~$\Sigma_f$.

\begin{defi}\label{defi.espace-vertebre}
On appelle espace G-tropical \emph{vertébré}\index{espace G-tropical vertébré}
la donnée d'un espace G-tropical chordé~$X$ 
et d'un entier~$n$ tel que $d_\trop(X)\leq n$,
tels que pour toute carte G-tropicale $(V,f)$ sur~$X$,
le sous-espace $\Sigma_f^{(n)}$ soit squelettique. 
\end{defi}

Avec ces notations, on dira plutôt que $X$ est un espace
G-tropical vertébré de niveau~$n$.

\begin{exem}
Soit $X$ un espace paralinéaire, considéré
comme espace G-tropical chordé dont les parties squelettiques
sont les parties paralinéaires (exemple~\ref{exem.paralineaire-chorde}).
Soit $n$ un entier naturel tel que $\dim(X)\leq n$.
Alors $d_\trop(X)=\dim(X)\leq n$ et 
pour tout domaine~$V$ de~$X$ et toute application paralinéaire
$f\colon V\to\R^N$, le sous-espace $\Sigma_f^{(n)}$
de~$V$ est paralinéaire, donc squelettique.
Cela munit~$X$ d'une structure d'espace G-tropical vertébré
de niveau~$n$.
\end{exem}

\subsection{}
Soit $V$ un domaine
d'un espace G-tropical vertébré~$X$ de niveau~$n$.
L'espace G-tropical~$V$ hérite d'une structure d'espace chordé 
et vérifie $d_\trop(V)\leq n$. 
Muni de cet entier~$n$, c'est un espace vertébré de niveau~$n$.

Noter que dans la définition d'un espace G-tropical vertébré, 
le niveau~$n$ est une donnée de la structure d'espace vertébré,
et non une propriété de l'espace chordé sous-jacent.
Par exemple, lorsqu'un domaine~$V$ vérifie $d_\trop(V)<n$,
il sera important de le considérer comme espace vertébré de niveau~$n$,
même si sa structure n'est pas très intéressante parce que 
ses parties $\Sigma_f^{(n)}$  sont vides.

Dans un espace vertébré, les parties $\Sigma_f^{(n)}$ sont squelettiques;
elles sont en particulier paralinéaires, et ce simple fait a déjà
des conséquences remarquables.

\begin{prop}\label{prop.image-squelette-pur}
Soit $X$ un espace G-tropical et $f\colon X\to\R^m$ une application
paralinéaire.
On suppose que pour tout domaine~$V$ de~$X$, 
le sous-espace $\Sigma_{f|_V}^{(n)}$ de~$V$ est paralinéaire;
c'est notammment le cas si $X$ est un espace vertébré de niveau~$n$.

\begin{enumerate}
\item Le sous-espace paralinéaire $\Sigma_f^{(n)}$ de $X$
est purement de dimension $n$, 
et l'application $f|_{\Sigma_f^{(n)}}$ est une immersion  par morceaux. 

\item Pour tout domaine $V$ de $X$, l'espace paralinéaire
$\Sigma_{f|_V}^{(n)}$ est égal à $(\Sigma_f^{(n)}\cap V) ^{(n)}$. 

\item Si $X$ est compact alors $f(X)^{(n)}=f(\Sigma_f^{(n)})$. 
\end{enumerate}
\end{prop}
\begin{proof}
\begin{enumerate}
\item
Par hypothèse, $\Sigma_f^{(n)}$ est une partie paralinéaire de~$X$
et l'application $f|_{\Sigma_f^{(n)}}\colon \Sigma_f^{(n)}\to\R^m$
est paralinéaire.

Rappelons d'abord que pour tout domaine compact~$V$ de~$X$,
les images 
$f(V)$ et $f(\Sigma_f^{(n)}\cap V)$ sont des polyèdres
compacts de~$\R^m$ tels que $f(\Sigma_f^{(n)}\cap V)\subset f(V)$,
et ces polyèdres sont de dimension au plus~$n$.

Soit $x$ un point de~$\Sigma_f^{(n)}$ 
et soit $V$ un domaine compact de~$X$ qui est un voisinage de~$x$.
L'application $\Sigma_f^{(n)}\cap V \to \R^m$ déduite de~$f$
est paralinéaire ; pour démontrer que $\Sigma_f^{(n)}$
est de dimension~$n$ en~$x$, il suffit de prouver
que le polyèdre compact $f(\Sigma_f^{(n)}\cap V)$ est de dimension~$n$.
Notons~$P$ ce polyèdre; 
pour tout  nombre réel $\eps>0$,
posons $P_\eps=P+\mathopen]-\eps,\eps\mathclose[^m$
et notons $V_\eps$ l'ensemble des points~$y$ de~$V$ tels
que $f(y)\not\in P_\eps$.
Par construction, $V_\eps$ est un domaine de~$V$, fermé, donc compact,
qui ne rencontre pas~$\Sigma_f^{(n)}$. Par suite, 
la dimension tropicale de~$f$ en tout point de~$V_\eps$ est~$<n$,
de sorte que $f(V_\eps)$ est un polyèdre compact de~$\R^m$
de dimension~$<n$. Par conséquent, on a $f(V)^{(n)}\subset P_\eps$.
En passant à la limite lorsque $\eps$ tend vers~$0$,
on en déduit que $f(V)^{(n)}\subset P$. 
Comme $f$ est de dimension tropicale~$n$ en~$x$ et comme $V$
est un voisinage de~$x$, le polyèdre~$f(V)$ est de dimension~$n$.
Par conséquent, on a $f(V)^{(n)}\neq \emptyset$, 
de sorte que $P$ est de dimension~$n$.

Il en résulte que $\Sigma_f^{(n)}$ est purement de dimension~$n$
et que l'application de~$\Sigma_f^{(n)}$ dans~$\R^m$
déduite de~$f$ est une immersion par morceaux.

\item
Soit $V$ un domaine de $X$. Il résulte immédiatement de la définition que 
$\Sigma_{f|_V}^{(n)}$ est contenu dans $\Sigma_f^{(n)}$.
D'après ce qui précède, la partie paralinéaire $\Sigma_{f|_V}^{(n)}$ est 
purement de dimension~$n$; 
elle est donc contenue dans $(\Sigma_f^{(n)}\cap V)^{(n)}$. 

Inversement, soit $x$ un point de $(\Sigma_f^{(n)}\cap V)
^{(n)}$ et soit $U$ un domaine compact de $V$ qui est un voisinage de~$x$ 
dans~$V$. Comme $x$ appartient à  $(\Sigma_f^{(n)}\cap V) ^{(n)}$, 
le domaine~$U$ contient
une partie paralinéaire (non vide) $P$ de dimension~$n$ 
de $\Sigma_f^{(n)}\cap V$. 
La restriction de $f$ à~$\Sigma_f^{(n)}$ étant une immersion par morceaux,
d'après~a),  on en déduit que $f(P)$ est de dimension~$n$.
En particulier, la dimension tropicale de~$f|_{\Sigma_f^{(n)}\cap V}$
est égale à~$n$, c'est-à-dire 
$x\in \Sigma_{f|_V}^{(n)}$. Ceci achève de prouver~b). 

\item
Puisque $X$ est compact, $f(\Sigma_f^{(n)})$ et $f(X)$
sont des polyèdres compacts de~$\R^m$
tels que $f(\Sigma_f^{(n)})\subset f(X)$.
Puisque la restriction de $f$ à~$\Sigma_f^{(n)}$ 
est une immersion par morceaux, $f(\Sigma_f^{(n)})$
est purement de dimension $n$, si bien que
$f(\Sigma_f^{(n)})\subset f(X)^{(n)}$.

Réciproquement, soit $s$ un point de $f(X)^{(n)}$, 
soit $P$ un voisinage paralinéaire compact de~$s$ dans~$f(X)$, 
et soit $V$ le domaine compact $f^{-1}(P)$ de~$X$; on a $f(V)=P$. 
Comme $V$ est compact, la dimension de~$f(V)$
est la borne supérieure des dimensions tropicales de~$f|_V$
en chacun des points de~$V$.
Puisque $s\in f(X)^{(n)}$, le polyèdre~$P$ est de dimension~$n$, 
si bien qu'il existe un point $x\in V$ en lequel
 $f|_V$ est de dimension tropicale $n$.
En particulier, $x\in \Sigma_f^{(n)}$;
comme $f(x)\in P$, cela entraîne que $P$  rencontre $f(\Sigma_f^{(n)})$. 
par suite, 
le point $s$ est adhérent à $f(\Sigma_f^{(n)})$. Ce dernier étant compact, 
il s'ensuit que $s$ appartient à $f(\Sigma_f^{(n)})$. \qedhere
\end{enumerate}
\end{proof}

\begin{prop}\label{prop.pur-n-squelette}
Soit $X$ un espace G-tropical vertébré de niveau~$n$ et 
soit $\Sigma$ une partie paralinéaire de~$X$ 
qui est purement de dimension $n$.
Il existe un G-recouvrement $(\Sigma_i)$ de~$\Sigma$ par des parties
paralinéaires compactes et, pour tout~$i$,
une carte G-tropicale~$f_i$ sur un domaine compact~$V_i$ 
telle que $\Sigma_i\subset \Sigma_{f_i}^{(n)}$.
En particulier, $\Sigma$ est squelettique.
\end{prop}
\begin{proof}
Par définition d'une partie paralinéaire,
elle  possède  un G-recouvrement par des parties fidèlement paralinéaires
compactes~$\Sigma_i$.
Puisque $\Sigma$ est purement de dimension~$n$,
la famille $(\Sigma_i^{(n)})$ est encore un G-recouvrement de~$\Sigma$.
D'après le lemme~\ref{lemme-proprietes-paralineairesX}, \emph e),
il existe pour tout~$i$ un domaine compact~$W_i$ de~$X$
tel que $\Sigma_i^{(n)}=\Sigma\cap W_i$.

Choisissons alors une fonction~$f_i$ dans $\Lambda(\Sigma_i^{(n)})$ 
et un domaine de définition~$V_i$ de~$f_i$, que l'on peut supposer compact. 
Comme $\Sigma_i^{(n)}$ est purement de dimension~$n$ 
et comme $f_i$ induit un homéomorphisme de $\Sigma_i^{(n)}$ sur son image, 
la dimension tropicale de~$f_i$ en tout point de $\Sigma_i^{(n)}$ vaut~$n$.
Par conséquent, $\Sigma_i^{(n)}$ est une partie paralinéaire 
de~$\Sigma_{f_i}^{(n)}$. 

La seconde assertion 
résulte de la proposition~\ref{prop.squelettique-paralineaireX}, \emph b),
car les~$\Sigma_{f_i}^{(n)}$ sont des parties squelettiques
(axiome~\emph e)),
\end{proof}

\begin{defi}\label{defi.riche}
Soit $X$ un espace G-tropical vertébré de niveau~$n$.
\begin{enumerate}
\item On dit que $X$ est \emph{riche}\index{espace G-tropical vertébré riche}
si pour toute carte G-tropicale $f$ définie
sur un domaine compact~$V$ de~$X$, il existe
une partie squelettique compacte~$\Sigma$ de~$V$
telle que $f(\Sigma)=f(V)$.
\item On dit que $X$ est \emph{très riche} s'il est riche
et si toute partie squelettique compacte est contenue
dans une partie squelettique purement de dimension~$n$.
\end{enumerate}
\end{defi}

Un domaine d'un espace riche est riche.

Une partie ouverte d'un espace très riche est très riche.

Un espace paralinéaire de niveau~$n$ est riche ; 
il est très riche si et seulement
s'il est purement de dimension~$n$.

\begin{exem}
Soit $k$ un corps complet pour une valeur absolue
non archimédienne et soit $X$ un espace $k$-analytique, 
topologiquement séparé, et purement de dimension~$n$.
Dans la seconde partie de ce livre, nous allons le munir d'une structure
d'espace G-tropical vertébré. Dans le but d'éclairer les constructions
et définitions de la présente partie, indiquons ici quelques
aspects de cette structure à venir.

Les domaines de~$X$, comme espace G-tropical, sont ses domaines analytiques.

Pour définir ses fonctions paralinéaires et ses fonctions affines,
on procédéra à rebours de l'axiomatisation générale en commençant
par définir les fonctions affines sur 
l'espace G-tropical~$X$: ce sont celles qui sont localement 
combinaisons $\Q$-linéaires 
de fonctions de la forme $\log \abs f$, 
où $f$ est une fonction analytique inversible sur un domaine analytique de~$X$.
Les fonctions paralinéaires sur l'espace G-tropical~$X$
sont alors les fonctions qui sont G-localement affines.

Parmi toutes les cartes G-tropicales sur~$X$, il suffira de considérer
celles de la forme $f_\trop\colon V\to \R^n$
où $V$ est un domaine analytique de~$X$, $f\colon V\to\gm^n$
est un morphisme d'espaces analytiques (nous appellerons \emph{moment}
de tels morphismes),
et $f_\trop= (\log \abs{f_1},\dots,\log\abs{f_n})$.

Les parties squelettiques élémentaires de~$X$
sont les parties de la forme~$\Sigma_f=f^{-1}(S_n)$, 
où $V$ est un domaine de~$X$ et $f\colon V\to\gm^n$ est un moment
et $S_n$ est le squelette standard de~$\gm^n$.
Alors, l'ensemble $(\Sigma_f)^{(n)}$ n'est autre que l'ensemble
$\Sigma_{f_\trop}^{(n)}$.

Si $X$ est un domaine analytique d'un bon espace,
en particulier, si $X$ est sans bord,
nous démontrerons que l'espace G-tropical associé à~$X$
est décent (définition~\ref{defi.espace-decent}
et proposition~\ref{prop.presque-bon-decent})
et localement différentiellement séparé
(corollaire~\ref{coro.berk-local-diff-sep}).

Lorsque $X$ erst purement de dimension~$n$,
nous démontrerons que l'espace G-tropical associé à~$X$
est riche (corollaire~\ref{coro.riche}), 
et même très riche si $X$ est sans bord (théorème~\ref{theo.berk-tres-riche}).
\end{exem}

\begin{rema}\label{rema.topologie-formes-squelettes}
Soit $X$ un espace G-tropical vertébré de niveau~$n$.
Dans ce cas, on peut donner une description squelettiques 
de la topologie sur l'espace des G-formes.

Soit $\mathscr D$ un jeu de cartes G-couvrant sur~$X$.
Soit $\Sigma$ une partie squelettique de~$X$.
Pour tout élément~$m\in\N\cup\{+\infty\}$,
l'application de restriction $\mathscr A^{p,q}_{\mathscr D, \groth}(X)
\to \mathscr A^{p,q}_\groth(\Sigma)$ est continue lorsque
ces deux espaces sont munis de leurs topologies~$\mathscr C^m$.

Nous allons démontrer que \emph{lorsque l'espace~$X$ est riche,
la topologie~$\mathscr C^m$ de $\mathscr A^{p,q}_{\mathscr D, \groth}(X)$
est la topologie localement convexe
la plus fine qui rende toutes ces applications continues.}
On peut même se limiter aux parties squelettiques~$\Sigma$
qui sont des cellules compactes compactes, 
contenues dans le domaine de définition~$V$
d'une carte~$f$ appartenant à~$\mathscr D$,
et telles que $f|_\Sigma$ soit injective.
A fortiori, on peut se limiter aux parties squelettiques compactes de~$X$.

En effet, soit $W$ un domaine compact d'une carte $(V,f)$ 
appartenant à~$\mathscr D$ et soit $r$ un entier tel que $r\leq m$.
Comme $X$ est supposé riche, il existe une partie
squelettique~$\Sigma$ de~$W$, compacte, telle que $f(\Sigma)=f(W)$.
L'application $f\colon \Sigma \to f(W)$ possède G-localement
des sections. Leurs images définissent une famille finie $(\Sigma_i)$
de parties squelettiques compactes de~$\Sigma$
telles que $\bigcup_i f(\Sigma_i)=f(W)$
et telles que $f|_{\Sigma_i}$ est injective pour tout~$i$.
Quitte à raffiner la famille~$(\Sigma_i)$, on peut supposer qu'elle
est constituée de cellules compactes.
Choisissons des normes~$\mathscr C^r$ standard sur $f(W)$ et les~$\Sigma_i$.
Pour toute forme~$\alpha$ sur~$f(W)$ et tout entier~$r$
tel que $r\leq m$, on a une majoration uniforme 
\[ \norm{\alpha}_{f(W),r} \ll \sup_i \norm{f^*\alpha|_{\Sigma_i}}_{\Sigma_i,r}.\]
L'assertion en résulte.
\end{rema}

\section{Supports forts}

Soit $X$ un espace G-tropical vertébré de niveau~$n$.

\subsection{}\label{sss.support-fort} 
Soit $\omega$ une G-forme de type~$(p,q)$ sur~$X$,
où $p$ et~$q$ sont deux entiers tels que $\max(p,q)=n$.
(En pratique, on aura $p\in\{n-1,n\}$ et $q=n$.)
On dit qu'une partie squelettique fermée~$\Sigma$ 
\emph{supporte fortement}~$\omega$ si elle contient le support de~$\omega$
et si, pour toute partie squelettique fermée~$\Tau$ contenant~$\Sigma$,
la forme $\omega|_\Tau$ est le prolongement par zéro de~$\omega|_\Sigma$.


Notons que dans ce cas, la partie squelettique fermée~$\Sigma$ 
supporte fortement 
$\omega\wedge\alpha$ pour toute G-forme~$\alpha$,
ainsi que les dérivées $\di\omega$, $\dc\omega$, etc.

Si $\Sigma$ est un support fort de~$\omega$,
alors $\Sigma\cap Y$ est un support fort de~$\omega|_Y$
pour tout domaine~$Y$ de~$X$.

Soit $(V_i)$ un G-recouvrement de~$X$. 
D'après~\S\ref{sss.support},
pour que $\Sigma$ supporte fortement~$\omega$,
il faut et il suffit 
que pour tout~$i$, $\Sigma\cap V_i$ 
supporte fortement~$\omega|_{V_i}$.

On peut vouloir résumer les propriétés précédentes en disant
que les G-formes fortement supportées par~$\Sigma$
constituent un sous-module différentiel gradué du G-faisceau 
$\mathscr F^{*,*}_{X_\groth}$.

\begin{rema}
Soit $\omega$ une G-forme de type~$(p,q)$ sur~$X$,
où $p$ et~$q$ sont deux entiers tels que $\max(p,q)=n$
et soit $\Sigma$ une partie squelettique fermée de~$X$
qui contient le support de~$\omega$.
Si la forme~$\omega$ est à coefficients tropicalement continus,
alors $\Sigma$ est un support fort de~$\omega$,
mais ce n'est pas vrai en général.

Par exemple, si $\theta$ est la fonction caractéristique de
l'origine dans l'espace G-tropical~$\R$,
la $(1,0)$-forme $\omega=\theta \di x$ a pour support~$\Sigma=\{0\}$,
qui est une partie squelettique fermée de~$\R$.
Ce n'en est pas un support fort, car $\omega|_\Sigma=0$,
dont le prolongement à~$\R$ est la forme nulle,
mais $\omega\neq0$.
\end{rema}

\begin{prop}\label{prop.support-fort-pur}
Soit $\omega$ une G-forme de type~$(p,q)$ sur~$X$,
où $p$ et $q$ sont des entiers tels que $\max(p,q)=n$
et soit $\Sigma$ une partie squelettique fermée de~$X$
qui supporte fortement~$\omega$.
Alors, $\Sigma^{(n)}$ est une partie squelettique fermée de~$X$
qui supporte fortement~$\omega$.
\end{prop}
\begin{proof}
Comme $\Sigma^{(n)}$ est une partie paralinéaire fermée de~$\Sigma$,
c'est une partie squelettique fermée de~$X$.

Soit $T$ une partie squelettique fermée de~$X$ qui contient~$\Sigma^{(n)}$;
démontrons que $\omega|_T$ est le prolongement par zéro de~$\omega|_{\Sigma^{(n)}}$. Quitte à remplacer~$T$ par $T\cup\Sigma$, 
on suppose que $T$ contient~$\Sigma$.  
Quitte à raisonner G-localement, on peut supposer que $T$
possède une décomposition cellulaire adaptée à~$\Sigma$,
et il s'agit de démontrer 
la restriction de~$\omega$ à toute cellule fermée qui n'est pas contenue
dans~$\Sigma^{(n)}$ est nulle.
C'est en effet le cas d'une cellule qui n'est pas dans~$\Sigma$,
par définition d'un support fort, et d'une cellule qui n'est
pas dimension~$n$, car la restriction d'une $(p,q)$-forme
à une telle cellule est nulle.

Prouvons maintenant que le support de~$\omega$
est contenu dans~$\Sigma^{(n)}$.
Soit $x$ un point de~$\Sigma^{(n)}$; prouvons que~$\omega$
est nulle au voisinage de~$x$.
L'assertion est G-locale au voisinage de~$x$ ; on peut donc supposer
que $X$ est compact,
que $\omega$ est G-tropicalisée par une carte~$f$ de domaine~$X$
et que $\Sigma^{(n)}=\emptyset$, c'est-à-dire que $\dim(\Sigma)<n$;
il s'agit alors de démontrer que $\omega=0$.

Notons~$P$ le polyèdre compact~$f(X)$ et
soit $\alpha$ une $(p,q)$-forme sur~$P$ 
telle que $\omega=f^*\alpha$.
Par hypothèse, $\omega|_{\Sigma\cup\Sigma_f^{(n)}}$
est le prolongement par zéro de~$\omega|_\Sigma$.
Soit $\mathscr C$ une décomposition cellulaire de~$\Sigma\cup\Sigma_f^{(n)}$
compatible à~$\Sigma$  et~$\Sigma_f^{(n)}$.
Soit $C$ une $n$-cellule de~$\mathscr C$; comme $\dim(\Sigma)<n$,
$\mathring C$ ne rencontre pas~$\Sigma$, donc $\omega|_C=0$
par définition du prolongement par zéro.
Comme $\Sigma_f^{(n)}$ est purement de dimension~$n$, 
cela entraîne que $\omega|_{\Sigma_f^{(n)}}=0$.
La composition de l'injection de~$\Sigma_f^{(n)}$ dans~$X$ 
et de~$f$ se factorise par~$f(\Sigma_f^{(n)})=P^{(n)}$ 
d'après la proposition~\ref{prop.image-squelette-pur}.
D'après la proposition~\ref{prop.nullite-formes}, 
on a $\alpha|_{P^{(n)}}=0$.
Comme $\max(p,q)=n$, on a donc $\alpha|_P=0$, d'où $\omega=0$.
\end{proof}


\begin{lemm}\label{lemm.Sigmaf-support-fort}
Si la forme~$\omega$ est G-tropicalisée par
une carte G-tropicale~$f$ de domaine~$X$, 
alors $\Sigma_f^{(n)}$ supporte fortement~$\omega$.
\end{lemm}
\begin{proof}
Le support de~$\omega$ est contenu dans l'ensemble
des points~$x\in X$ en lesquels $f$ est de dimension
tropicale~$\geq n$ qui, par définition, est~$\Sigma_f^{(n)}$.

Soit $\Tau$ un sous-espace paralinéaire fermé de~$X$
qui contient~$\Sigma_f^{(n)}$. Démontrons que $\omega|_\Tau$
est le prolongement par zéro de~$\omega|_{\Sigma_f^{(n)}}$.
Soit $\mathscr C$ une décomposition cellulaire en $f$-cellules
de~$\Tau$ qui est compatible à~$\Sigma_f^{(n)}$. Il suffit
de prouver que pour toute cellule~$C$ de~$\mathscr C$
telle que $\mathring C$  est disjoint de~$\Sigma_f^{(n)}$,
on a $\omega|_C=0$.

Par hypothèse, $f$ est de dimension tropicale~$<n$ en tout point
de~$\mathring C$, donc il existe une cellule~$D$, compacte,
d'intérieur non vide, de~$\mathring C$ telle que $\dim(f(D))<n$.
Comme $C$ est une $f$-cellule, on  $\dim(f(C))=\dim(f(\mathring C))=\dim(f(D))$.
D'après la proposition~\ref{prop.nullite-formes},
on a donc $\omega|_{C}=0$.

Cela démontre que $\Sigma_f^{(n)}$ supporte fortement~$\omega$.
%
\end{proof}

\begin{prop}\label{lemm.squelette-support-fort}
Soit $n$ un entier et soit $X$ un espace G-tropical vertébré
de niveau~$n$. Soit $(p,q)$ un couple d'entiers naturels tel que 
$\max(p,q)=n$ et
soit $\omega$ une G-forme de type~$(p,q)$ sur~$X$.
Supposons que $X$ est paracompact,
ou que le support de~$\omega$ est compact.
Il existe alors une partie squelettique fermée de~$X$
qui supporte fortement~$\omega$.
\end{prop}
\begin{proof}
Lorsque $X$ est paracompact,
on considère un G-recouvrement localement fini de~$X$ 
par des domaines compacts qui G-tropicalisent~$\omega$;
lorsque $\omega$ est à support compact, on considère un G-recouvrement
fini d'un voisinage du support de~$\omega$ par de tels domaines.
Pour tout tel domaine~$V$, 
soit $f_V$ une carte G-tropicale sur~$V$ G-tropicalisant~$\omega|_V$;
d'après l'exemple précédent, 
$\Sigma_{f_V}^{(n)}$ supporte fortement~$\omega|_V$.
Soit $\Sigma$ la réunion des parties squelettiques~$\Sigma_{f_V}^{(n)}$.
D'après la proposition~\ref{prop.squelettique-paralineaireX},
c'est une partie squelettique fermée de~$X$;
de plus, pour tout domaine~$V$ du G-recouvrement donné,
la partie squelettique fermée $\Sigma\cap V$ de~$V$ 
supporte fortement $\omega|_V$,
puisque $\Sigma\cap V$ contient~$\Sigma_{f_V}^{(n)}$.
Par suite, $\Sigma$ supporte fortement~$\omega$.
\end{proof}
 
 
\begin{prop}\label{prop.restr-support-fort}
Soit~$n$ un entier et soit~$X$ un espace G-tropical vertébré de niveau~$n$.
Soit~$(p,q)$ un couple d'entiers naturels tel que $\max(p,q)=n$
et soit~$\omega$ une G-forme de type~$(p,q)$.
Soit $\Sigma$ un support fort de~$\omega$; alors le support de~$\omega|_\Sigma$
est égal au support de~$\omega$.
\end{prop}
\begin{proof}
La propriété étant G-locale, on peut supposer que $X$ est 
est compact et le domaine d'une carte G-tropicale~$f$ 
qui G-tropicalise~$\omega$;
soit $\alpha$ une $(p,q)$-forme telle que $\omega=f^*\alpha$.

Soit $x\in X$.
Si $x\in\supp(\omega|_\Sigma)$, alors $x\in\supp(\omega)$.
Supposons inversement que $x\in\supp(\omega)$ et prouvons
que $x\in\supp(\omega|_\Sigma)$.

Soit $\Tau=\Sigma\cup\Sigma_f^{(n)}$. Par définition d'un support
fort, le support de~$\omega|_\Tau$ est égal au support de~$\omega|_\Sigma$.
De même, comme $\Sigma_f^{(n)}$ est un support fort de~$\omega$,
le support de~$\omega|_\Tau$ est égal au support de~$\omega|_{\Sigma_f^{(n)}}$.
Il s'agit donc de prouver que $x\in\supp(\omega|_{\Sigma_f^{(n)}})$.

Le point~$f(x)$ appartient au support de~$\alpha$ 
dans~$f(X)$ (corollaire~\ref{coro.support}), 
support qui n'est autre
que le support de la restriction de~$\alpha$ à~$f(X)^{(n)}$, 
puisque $\alpha$ est une $(p,q)$-forme.
L'application paralinéaire~$f$ induit une immersion par morceaux 
de~$\Sigma_f^{(n)}$ sur son image $f(\Sigma_f^{(n)})=f(X)^{(n)}$
(proposition~\ref{prop.image-squelette-pur}),
par laquelle $\omega|_{\Sigma_f^{(n)}}$ 
est l'image réciproque de~$\alpha$. 
Comme cette immersion par morceaux possède G-localement des sections,
le support de~$\omega|_{\Sigma_f^{(n)}}$ est l'image réciproque
du support de~$\alpha$; il contient en particulier le point~$x$.
\end{proof}

\section{Cartes agréables, polyèdres et cônes tropicaux}

\begin{defi}\label{defi.f-cellule-trop}
Soit $X$ un espace G-tropical vertébré de niveau~$n$, 
soit $f\colon X\to E$ une carte G-tropicale  et soit $C$ une partie de~$X$.
On dit $C$ est une \emph{$f$-cellule} si c'est
une $f$-cellule de l'espace paralinéaire~$\Sigma_f^{(n)}$.
\end{defi}

Une $f$-cellule de~$X$ est une partie squelettique de~$X$;
elle est connexe et non vide, donc est contenue dans
une unique composante connexe de~$X$.

\begin{rema}
Soit $C$ une partie de~$X$ et soit $f\colon X\to E$ une carte
tropicale.
Si $f$ induit un homéomorphisme de~$C$ sur une cellule 
\emph{de dimension~$n$} de~$E$, alors 
$C$ est contenue dans~$\Sigma_f^{(n)}$; c'est alors une $f$-cellule.
\end{rema}

\begin{defi}\label{defi.carte-agreable}
On dit qu'une carte G-tropicale $f\colon X\to E$ 
est \emph{agréable}\index{carte G-tropicale agréable}
si $f(\Sigma_f^{(n)})$ est paralinéaire et 
si l'application paralinéaire $f|_{\Sigma_f^{(n)}}$ est propre sur son image.
\end{defi}
C'est en particulier le cas si $f|_{\Sigma_f^{(n)}}$ est propre,
en particulier si $\Sigma_f^{(n)}$ possède une décomposition cellulaire 
finie en $f$-cellules.

Lorsque $f$ est une carte agréable, nous noterons~$\Pi_f$
l'ensemble $f(\Sigma_f^{(n)})$ et l'appellerons le \emph{polyèdre
tropical}\index{polyèdre tropical}
de la carte G-tropicale agréable~$f$.
Il est purement de dimension~$n$ et 
l'application~$f$ induit une immersion par morceaux
de~$\Sigma_f^{(n)}$ sur~$\Pi_f$.

\subsection{}\label{exem.compact-agreable}
Si $X$ est compact, toute carte G-tropicale est agréable
et l'on a $\Pi_f=f(X)^{(n)}$
(proposition~\ref{prop.image-squelette-pur}).

\begin{prop}\label{prop.agreable-restriction}
Soit $X$ un espace G-tropical vertébré de niveau~$n$, 
soit $f\colon X\to E$ une carte G-tropicale agréable et 
soit $P$ un sous-espace paralinéaire~$P$ de~$E$.
Considérons 
la carte G-tropicale $f_P\colon f^{-1}(P)\to E$ 
sur le domaine $f^{-1}(P)$ (considéré comme espace G-tropical
vertébré de niveau~$n$) déduite de~$f$ par restriction.
La carte~$f_P$ est agréable et 
on a les relations
\begin{gather}
\Sigma_{f_P}^{(n)} = (\Sigma_f^{(n)}\cap f^{-1}(P))^{(n)} \\
\Pi_{f_P}= (\Pi_f\cap P)^{(n)} .
\end{gather}
\end{prop}
\begin{proof}
L'égalité $\Sigma_{f_P}^{(n)} = (\Sigma_f^{(n)}\cap f^{-1}(P))^{(n)} $
est l'assertion~\emph b) de la proposition~\ref{prop.image-squelette-pur}.
On a donc
\[ f(\Sigma_{f_P}^{(n)}) = f((\Sigma_f^{(n)}\cap f^{-1}(P))^{(n)}).\]
Comme $f|_{\Sigma_f^{(n)}}$ est une immersion par morceaux compacte
sur son image et comme $\Sigma_f^{(n)}\cap f^{-1}(P)$ 
en est une partie paralinéaire, il résulte de
la proposition~\ref{prop.sections-G-locales} que
\[ f((\Sigma_f^{(n)}\cap f^{-1}(P))^{(n)})
= f(\Sigma_f^{(n)}\cap f^{-1}(P))  ^{(n)}
= (f(\Sigma_f^{(n)})\cap P)  ^{(n)}= (\Pi_f\cap P)^{(n)}.
\]
Comme $\Sigma_{f_P}^{(n)}$ est une partie fermée de~$\Sigma_f^{(n)}$,
l'application $f$ induit une immersion par morceaux compacte
de~$\Sigma_{f_P}^{(n)}$ sur son image.
D'après ce qui précède, cette image est égale à $f(\Sigma_{f_P}^{(n)})$,
donc est une partie paralinéaire de~$E$.
Ainsi, $f_P$ est une carte G-tropicale agréable et son 
polyèdre tropical est égal à 
$\Pi_{f_P}=f(\Sigma_{f_P}^{(n)})=(\Pi_f\cap P)^{(n)}$.
\end{proof}

\subsection{}
Soit $X$ un espace $G$-tropical vertébré de niveau $n$ 
et soit $f\colon X\to E$ une carte tropicale. 
Soit $x$ un point de $\Sigma^{(n)}_f$.

Considérons un voisinage ouvert paracompact~$W$ de~$x$.
Alors $\Sigma^{(n)}_f\cap W= \Sigma^{(n)}_{f|_{W}}$;
considérons-en une décomposition cellulaire~$\mathscr C$
en $f$-cellules. 

Soit $V$ la réunion des $n$-cellules de~$\mathscr C$ qui contiennent~$x$.
C'est un voisinage paralinéaire compact de~$x$ dans~$\Sigma^{(n)}_f$
et la restriction de~$f$ à~$V$
est une immersion linéaire par morceaux de~$V$ dans~$E$
qui identifie chaque cellule~$C$ à une cellule forte de~$E$.
Par construction, $x$ est le seul antécédent de $f(x)$ sur $V$. 
Soit $W'$ un autre voisinage ouvert paracompact de $X$, soit $\mathscr C'$ une 
décomposition cellulaire de $\Sigma_f^{(n)}\cap W'$ et soit $V'$ la réunion des
$n$-cellules de $\mathscr C'$ contenant $x$. L'intersection $V\cap V'$ est un voisinage 
paralinéaire compact de~$x$ dans~$V$. par propreté topologique de $f$ il existe un voisinage
paralinéaire compact~$U$ de~$f(x)$ dans~$E$ tel que
$f^{-1}(U)\subset V\cap V'$. Il s'ensuit que les polyèdres~$f(V)$ et~$f(V')$ 
coïncident au voisinage de~$f(x)$ dans~$E$. Leur germe
commun est le germe d'un unique cône de sommet~$f(x)$.
On dit que c'est le \emph{cône tropical} de~$f$ en~$x$
et on le note~$\Pi_{f,x}$.\index{cône tropical}\label{defi.cone-tropical}

\subsection{}
Soit $U$ un voisinage domanial 
compact de~$x$, assez petit
de sorte que $U\cap \Sigma_f^{(n)}\subset V$.
Pour toute $n$-cellule~$C$ de~$\mathscr C$ qui contient~$x$,
l'intersection 
$U\cap C$ est un voisinage de~$x$ dans~$C$,
donc $f(U\cap V)$ est un voisinage de~$f(x)$ dans~$f(C)$.
Comme la famille des~$f(C)$  est un G-recouvrement de~$f(V)$,
on en déduit que $f(U\cap V)$ est un voisinage de~$f(x)$
dans~$f(V)$, purement de dimension $n$. 
On a alors

\[f(U)^{(n)} =f(\Sigma^{(n)}_{f|_U}) \subset f(U\cap \Sigma^{(n)}_f)
=f(U\cap V)\subset f(U)^{(n)},\]
où la première égalité provient de la proposition~\ref{prop.image-squelette-pur} \emph c) 
et la dernière du fait que $f(U\cap V)$ est purement de dimension $n$. 
Il en résulte que $f(U^{(n)})=f(U\cap V)$ ; par conséquent, le germe de polyèdre $f(U)^{(n)}$
en~$f(x)$ ne dépend pas de~$U$ dès que 
$U\cap \Sigma_f^{(n)}\subset V$, puisqu'il est alors égal
au germe de polyèdre de $f(V)$ en $f(x)$, c'est-à-dire encore à 
$\Pi_{f,x}$.

\section{Espaces G-tropicaux calibrés}

Soit $X$ un espace G-tropical vertébré de niveau~$n$.

\begin{defi}
On appelle \emph{$n$-calibrage} de~$X$ 
la donnée, pour toute 
partie squelettique~$\Sigma$ de~$X$,
d'un $n$-calibrage strictement positif~$\mu_{\Sigma}$ sur~$\Sigma$.
On exige en outre que si $T$ est une partie squelettique
contenue dans~$\Sigma$, on ait 
l'égalité
$\mu_T=\mu_\Sigma|_T$.

Un \emph{espace G-tropical calibré} de niveau~$n$ 
est un espace G-tropical vertébré de niveau~$n$ muni 
d'un $n$-calibrage.
\end{defi}

Dans les applications que nous avons en vue, le calibrage
est toujours G-localement constant.
 
\begin{exem}
Si $X$ est un espace paralinéaire calibré de niveau~$n$,
alors l'espace G-tropical vertébré associé est  naturellement
muni d'un $n$-calibrage.
\end{exem}

\subsection{}
Supposons que $X$  soit muni d'un $n$-calibrage.
Alors, tout domaine de~$X$ hérite d'un $n$-calibrage.

\begin{lemm}\label{lemm.gros-site}
Soit $\mathscr S$ un ensemble de parties squelettiques de~$X$
telle que toute partie squelettique de~$X$
soit G-recouverte par des parties paralinéaires d'éléments de~$\mathscr S$.
Pour tout $\Sigma\in\mathscr S$, supposons donné un $n$-calibrage~$\mu_\Sigma$
de~$\Sigma$ de sorte que pour tous~$\Sigma$ et~$\Tau\in\mathscr S$,
les calibrages $\mu_\Sigma$ et~$\mu_\Tau$ coïncident sur $\Sigma\cap\Tau$.
Il existe alors un unique $n$-calibrage de~$X$ qui coïncide avec
les calibrages donnés sur les parties squelettiques appartenant à~$\mathscr S$.
\end{lemm}
Un cas particulier important du lemme est celui où $\mathscr S$
est l'ensemble des parties squelettiques élémentaires de~$X$
et de ses domaines.
\begin{proof}
%
%
Soit $\Sigma$ une partie squelettique de~$X$.

Par hypothèse, elle possède un G-recouvrement de la forme~$(\Sigma_i)$,
où, pour tout~$i$, $\Sigma_i$ est une partie paralinéaire
d'un élément~$S_i$ de~$\mathscr S$. On a nécessairement
$\mu_\Sigma|_{\Sigma_i}= \mu_{S_i}|_{\Sigma_i}$.
Soit $i,j$ et des indices;
par hypothèse, les calibrages~$\mu_{S_i}$ et~$\mu_{S_j}$
coïncident sur~$S_i\cap S_j$;
par conséquent, 
les calibrages~$\mu_{S_i}|_{\Sigma_i}$ et $\mu_{S_j}|_{\Sigma_j}$
coïncident sur l'intersection $\Sigma_i\cap \Sigma_j$.
Il existe donc un unique calibrage~$\mu_\Sigma$ de~$\Sigma$
qui coïncide avec $\mu_{S_i}|_{\Sigma_i}$ sur~$\Sigma_i$, pour tout~$i$.

Ce calibrage~$\mu_\Sigma$ ne dépend pas du choix du G-recouvrement~$(\Sigma_i)$,
comme on le voit en considérant la réunion de deux G-recouvrements.

La famille de ces calibrages convient, et c'est la seule.
\end{proof}

\subsection{}
Soit $f\colon X\to E$ une carte G-tropicale agréable.
On munit alors le polyèdre~$\Pi_f$
du calibrage $f_*\mu_{\Sigma_f^{(n)}}$ image du calibrage
de~$\Sigma_f^{(n)}$ par l'immersion par morceaux propre $f|_{\Sigma_f^{(n)}}$;
on note $\mu_{\Pi_f}$ ce calibrage.

\begin{exem}\label{exem.agreable-restriction-calibrage}
Soit $f\colon X\to E$ une carte G-tropicale agréable.
Soit $P$ un sous-espace paralinéaire de~$\Pi_f$.
Posons $Y=f^{-1}(P)$ et posons $g=f|_Y$.
D'après la proposition~\ref{prop.agreable-restriction},
la carte $g=f|_Y$ est agréable, 
on $\Sigma_g^{(n)}=(\Sigma_f^{(n)}\cap f^{-1}(P))^{(n)}$
et $\Pi_g=P^{(n)}$. Le calibrage de~$\Sigma_g^{(n)}$
est la restriction de celui de~$\Sigma_f^{(n)}$.
Il découle alors des définitions
que le calibrage de~$\Pi_g$ est la restriction 
du calibrage de~$\Pi_f$.
\end{exem}

\begin{exem}\label{rema.trop-calibrage-projection}
Soit $f\colon X\to E$ une carte G-tropicale,
soit $p\colon E\to E'$ une application affine, et soit $g=p\circ f$.
Supposons que la carte G-tropicale~$g$ soit agréable.

Soit $\Pi'_f$ la réunion des $n$-cellules de~$E$
contenues dans~$\Pi_f$  qui ne sont pas contractées par~$p$.
Comme c'est la réunion des $n$-cellules d'une décomposition cellulaire
donnée de~$\Pi_f$ qui ne sont pas contractées par~$p$,
l'ensemble~$\Pi'_f$ est un sous-espace paralinéaire fermé de~$\Pi_f$.
Démontrons que $p$ induit une immersion
linéaire par morceaux \emph{compacte} de $\Pi'_f$ sur~$\Pi_g$
et que le calibrage canonique de~$\Pi_g$ 
est l'image de la restriction à~$\Pi'_f$
du calibrage canonique de~$\Pi_f$.

Comme $g=p\circ f$, la dimension tropicale de~$g$ est en tout point
inférieure à celle de~$f$ ; on a donc 
$\Sigma_g^{(n)}\subset \Sigma_f^{(n)}$.

Démontrons d'abord que $f_\trop^{-1}(\Pi'_f)=\Sigma_g^{(n)}$.
Soit~$x$ un point de~$\Sigma_f^{(n)}$ tel que $f(x)\in\Pi'_f$.
Soit $D$ une $n$-cellule contenue dans~$\Pi_f$ contenant~$f(x)$ et
non contractée par~$p$ ; comme $f$ est une immersion
par morceaux, on peut supposer qu'il existe une $n$-cellule~$C$ de~$\Sigma_f^{(n)}$ telle que $f$ induise un isomorphisme de~$C$ sur~$D$.
Comme $D$ n'est pas contractée par~$p$, la dimension tropicale de~$g$
en~$x$ est au moins égale à~$n$, donc $x\in\Sigma_g^{(n)}$.
Inversement, soit~$x$ un point de~$\Sigma_g^{(n)}$
et soit~$C$ une $n$-cellule de~$\Sigma_g^{(n)}$ dont
l'image par~$f$ est une $n$-cellule~$D$ de~$\Pi_f$;
comme $g|_C$ est injective, l'application~$p|_D$ est injective,
d'où $f(x)\in\Pi'_f$.

On a $\Pi_g=g(\Sigma_g^{(n)})=p(f(\Sigma_g^{(n)}))=p(\Pi'_f)$ et
l'application $p\colon \Pi'_f\to \Pi_g$ 
est une immersion linéaire par morceaux. Remarquons
qu'elle est compacte: en effet, soit $A$ une partie compacte de~$\Pi_g$,
alors $p^{-1}(A)\cap \Pi'_f=f(g^{-1}(A)\cap\Sigma_g^{(n)})$
donc est compacte puisque, $g$ étant agréable, 
$g^{-1}(A)\cap\Sigma_g^{(n)}$ est compact.

D'après la première partie, l'image par~$f$
du calibrage de $\Sigma_g^{(n)}$ est égale à la restriction
à~$\Pi'_f$ du calibrage de~$\Pi_f$.
Le calibrage de~$\Pi_g$ est donc l'image par~$p$ de
ce calibrage de~$\Pi'_f$.
\end{exem}

\subsection{}\label{defi.cone-tropical-calibre}
Soit $X$ un espace $G$-tropical calibré de niveau~$n$ 
et soit $f\colon X\to E$ une carte tropicale. 
Soit~$x$ un point de $\Sigma^{(n)}_f$. Nous allons calibrer
le cône tropical~$\Pi_{f,x}$ défini au~\S\ref{defi.cone-tropical}.

Soit $V$ un voisinage paralinéaire compact de~$x$ dans~$\Sigma_f^{(n)}$,
assez petit pour que le germe du cône~$\Pi_{f,x}$ 
en~$x$ soit le germe de~$f(V)$ en~$f(x)$.
On en déduit un germe de calibrage de~$\Pi_{f,x}$ en~$f(x)$
qui se prolonge de manière unique en un calibrage de ce cône
qui est adapté à une décomposition cellulaire conique,
et donc à toute décomposition cellulaire conique assez fine.
Ce calibrage ne dépend pas du choix de~$V$.

Si $U$ est un voisinage domanial compact de~$x$,
assez petit, le germe du polyèdre calibré~$\Pi_{f|_U}$
en~$f(x)$ coïncide avec le germe du cône tropical calibré~$\Pi_{f,x}$.

\section{Formes-volume sur un espace G-tropical vertébré}
\label{ss.forme-volume-Gtrop}

Soit $p$ et~$n$ des entiers naturels tels que $p\leq n$.
Dans ce paragraphe, nous définissons un faisceau de formes-volume 
$p$-dimensionnelles sur un espace G-tropical vertébré~$X$ de niveau~$n$.
Ses sections sont
localement de type Lebesgue sur une partie squelettique  
de dimension au plus~$p$.

On fixe une classe de coefficients~$\mathscr G$.

\subsection{}
Soit~$X$ un espace G-tropical vertébré de niveau~$n$.

Soit~$\mathscr P_p(X)$ l'ensemble des parties squelettiques fermées de~$X$
de dimension~$\leq p$;
ordonnons-le par l'inclusion.
Pour tout~$\Sigma\in \mathscr P_p(X)$, on dispose
de l'espace $\mathscr V^p(\Sigma)$ des formes-volume
$p$-dimensionnelles sur~$\Sigma$.
Les morphismes de prolongement par zéro font de
la famille~$(\mathscr V^p(\Sigma))_{\Sigma \in \mathscr P_p(X)}$
un système inductif filtrant ; notons~$\mathscr V_{\pf}^p(X)$ 
sa limite inductive. Observons que pour tout~$\Sigma$,
l'application canonique de $\mathscr V^p(\Sigma)$ 
dans $\mathscr V_\pf^p(X)$ est injective.

Soit~$V$ un domaine de~$X$. 
Soit~$\lambda \in \mathscr V_{\pf}^p(X)$, provenant d'une 
forme-volume~$\lambda_\Sigma\in \mathscr V^p(\Sigma)$, 
à coefficients dans~$\mathscr F$,
pour une certaine partie squelettique fermée~$\Sigma\in \mathscr P_p(X)$. 
L'image de~$\lambda_\Sigma|_{\Sigma \cap V}$ 
dans~$\mathscr V^p(V)$ ne dépend que de~$\lambda$, et
sera notée~$\lambda|_V$. 
Grâce à ces flèches de restriction, l'application
$V\mapsto \mathscr V^p_\pf(V)$ apparaît comme
un préfaisceau sur~$X_{\mathrm G}$ 
qu'on note $\mathscr V_{X_{\mathrm G},\pf}^{\mathscr F,p}$,
où plus simplement $\mathscr V_{X_\groth,\pf}^{p}$ s'il
ne peut y avoir d'ambiguïté quant à la classe~$\mathscr F$.

Ce préfaisceau 
a une structure naturelle de module sur le faisceau d'anneaux 
des fonctions numériques sur~$X$ dont la restriction à toute
partie squelettique appartient à~$\mathscr G$.

\begin{lemm}\label{mesure-polyedrale-separe}
Le préfaisceau~$\mathscr V_{X_{\mathrm G},\pf}^{p}$ est séparé.
\end{lemm}
\begin{proof}
Soit~$V$ un domaine de~$X$, 
soit~$(V_i)$ un G-recouvrement de~$V$ et soit 
$\lambda$ un élément de~$\mathscr V_{X_\groth,\pf}^{p}(V)$
tel que~$\lambda|_{V_i}=0$ pour tout~$i$. Démontrons que~$\lambda=0$. 

Choisissons une partie squelettique fermée $\Sigma$ de~$V$ 
telle que $\lambda$ provienne d'un élément~$\lambda_\Sigma$ 
de~$\mathscr V^p(\Sigma)$. 
Par hypothèse, la restriction de~$\lambda_\Sigma$ 
à~$\Sigma\cap V_i$ est nulle pour tout~$i$ ; comme les~$\Sigma\cap V_i $ 
constituent un G-recouvrement de~$\Sigma$, on a~$\lambda_\Sigma=0$, 
et donc~$\lambda=0$. 
\end{proof}

\subsection{}
On note $\mathscr V^{\mathscr G,p}_{X_{\mathrm G}}$ 
le G-faisceau associé au G-préfaisceau
$\mathscr V_{X_{\mathrm G},\pf}^{\mathscr G,p}$ 
et on dit que c'est le \emph{faisceau
des formes-volume $p$-dimensionnelles à coefficients dans~$\mathscr G$
sur~$X_{\mathrm G}$.}
C'est un faisceau en modules sur le faisceau d'anneaux 
des fonctions numériques sur~$X$ dont la restriction à toute
partie squelettique appartient à~$\mathscr G$.

On note $\mathscr V_{X}^{\mathscr G,p}$ 
le faisceau associé pour la topologie usuelle;
on dit que c'est le \emph{faisceau des formes-volume $p$-dimensionnelles 
à coefficients dans~$\mathscr G$ sur~$X$.}
\index{forme-volume sur un espace G-tropical vertébré}

Notons $\pi \colon X_{\mathrm G}\to X_{\mathrm {top}}$ 
le morphisme de sites canonique;  
il découle du lemme~\ref{mesure-polyedrale-separe} que 
$\mathscr{V}^{p}_{X}$ est le
sous-faisceau de $\pi_*\mathscr{V}_{X_\groth}^{p}$ 
constitué des sections qui appartiennent localement,
et non seulement G-localement, 
à~$\mathscr V_{X_{\mathrm G,\pf}}^{p}$. 

\begin{prop}\label{mes-fasc-paracompact}
Supposons que $X$ est paracompact.
Soit~$(X_i)$ un G-recouvrement de~$X$ ; pour tout~$i$, 
soit~$\lambda_i\in \mathscr V^{p}_\pf(X_i)$. 
Si~$\lambda_i|_{X_i\cap X_j}=\lambda_j|_{X_i\cap X_j}$ 
pour tout~$(i,j)$ il existe une unique
forme-volume~$\lambda\in \mathscr V^{p}_\pf(X)$ 
telle que $\lambda|_{X_i}=\lambda_i$ pour tout~$i$.
\end{prop}
\begin{proof}
L'unicité de~$\lambda$ provient du lemme~\ref{mesure-polyedrale-separe}
ci-dessus. 
En vertu de cette assertion d'unicité, l'existence peut être prouvée après
raffinement du recouvrement~$(Y_i)$. 
Puisque $X$ est paracompact, on peut donc supposer
que le recouvrement~$(Y_i)$ est localement fini. 
Pour tout~$i$, choisissons une partie squelettique
fermée~$\Sigma_i$ de~$Y_i$ telle que~$\lambda_i$ provienne d'un élément
$\lambda_{\Sigma_i}\in \mathscr V^p(\Sigma_i)$. 
Comme le G-recouvrement~$(Y_i)$ est localement fini, 
la réunion~$\Sigma$ des $\Sigma_i$ 
est une partie squelettique fermée de~$X$ 
(proposition~\ref{prop.squelettique-paralineaireX})
dont la famille~$(\Sigma_i)$ est un G-recouvrement. 
Les~$\lambda_{\Sigma_i}$ se recollent alors 
en une section~$\lambda_\Sigma\in \mathscr V^p(\Sigma)$.
Notons $\lambda$ la classe de~$\lambda_{\Sigma}$ dans~$\mathscr V^p(X)$;
alors~$\lambda|_{X_i}=\lambda_i$ quel que soit~$i$. 
\end{proof}

\begin{coro}
Soit $Y$ un domaine \emph{paracompact} de~$X$
et soit $\lambda\in\mathscr V^{\mathscr F,p}_{X_\groth}(Y)$
une forme-volume $p$-dimensionnelle sur~$Y$.
Il existe une partie squelettique fermée~$\Sigma$ de~$Y$ de dimension~$\leq p$
et une $p$-forme-volume~$\lambda_\Sigma$ sur~$\Sigma$ qui définit~$\lambda$.
\end{coro}
Il est vraisemblable que le corollaire ne vaille pas 
sans l'hypothèse que $Y$ est paracompact.

\subsection{}
Soit $Y$ un domaine fermé de~$X$; notons $i$
l'immersion canonique de~$Y$ dans~$X$.
Comme toute partie squelettique fermée de~$Y$ 
est une partie squelettique fermée de~$X$,
le prolongement par zéro définit un homomorphisme de G-préfaisceaux
de $i_*\mathscr V^p_{Y_\groth,\pf}$ dans~$\mathscr V^p_{X_\groth,\pf}$.
On en déduit un morphisme de G-faisceaux 
de  $i_*\mathscr V^p_{Y_\groth}$ dans~$\mathscr V^p_{X_\groth}$.
Si $\lambda$ est une forme-volume $p$-dimensionnelle sur~$Y$,
on notera $i_*\lambda$ la forme-volume $p$-dimensionnelle 
correspondante sur~$X$.

\subsection{}
On dit qu'une forme-volume $p$-dimensionnelle~$\lambda$ sur~$X$
est positive s'il existe un G-recouvrement~$(X_i)$ de~$X$
tel que, pour tout~$i$, $\lambda$ soit représentée
par une formule-volume positive sur une partie squelettique fermée
de~$X_i$.

Dans ce cas, 
toute forme-volume sur une partie squelettie fermée d'un domaine~$Y$
de~$X$ induisant la forme-volume $\omega|_Y$ est elle-même positive.

\section{Mesures associées à une forme-volume $p$-dimensionnelle}

Soit~$X$ un espace G-tropical vertébré de niveau~$n$
et soit $p$ un entier naturel tel que $p\leq n$.
Nous allons associer à toute forme-volume~$\lambda$,
$p$-dimensionnelle, à coefficients boréliens, sur~$X_{\mathrm G}$
une mesure de Borel régulière $\lambda_{\Bor}$ sur~$X$
et, lorsque cette mesure est localement finie,
une mesure de Radon~$\lambda_{\Rad}$ sur~$X$
telle que $\lambda_{\Bor}=\abs{\lambda_{\Rad}}$
qui est positive lorsque $\lambda$ est une forme-volume positive.

Soit $\mathscr G$ une classe de coefficients qui est contenue dans~$\mathscr B$.

\subsection{}
\label{sss.mesure-forme-volume}
Soit~$\Sigma$ et~$\Sigma'$ des parties squelettiques fermées de~$X$ 
telles que $\Sigma \subset \Sigma'$. Soit~$\lambda$ appartenant 
à~$\mathscr V^{\mathscr G,p}(\Sigma)$ 
et soit~$\lambda'$ son prolongement par zéro à $\Sigma'$. 
Notons $\mu$ et $\mu'$ les mesures de Borel sur~$\Sigma$ et~$\Sigma'$ 
associées à~$\lambda$ et~$\lambda'$; lorsqu'on les considère
comme mesures de Borel sur~$X$, elles sont égales.

Les mesure de Borel~$\mu$ et~$\mu'$ sont simultanément localement
finies ou non. Lorsqu'elles le sont, les formes-volume~$\lambda$
et~$\lambda'$ donnent lieu à des mesures de Radon~$\tilde\mu$
et $\tilde \mu'$ sur~$\Sigma$
et~$\Sigma'$ dont $\mu$ et $\mu'$ sont les valeurs absolues.
On a de plus $\tilde\mu=\tilde\mu'$ lorsque ces mesures de Radon
sont identifiées à des mesures de Radon sur~$X$.

Par conséquent, il existe  un unique morphisme
de faisceaux $\lambda\mapsto\lambda_{\Bor}$
sur~$X_{\mathrm G}$ qui, pour tout domaine~$Y$ de~$X$, 
toute partie squelettique fermée~$\Sigma$ de~$Y$
et toute forme-volume~$\lambda_\Sigma$ sur~$\Sigma$, $p$-dimensionnelle
et à coefficients dans~$\mathscr G$,
applique la classe de~$\lambda_\Sigma$ dans~$\mathscr V^{\mathscr G,p}(Y)$
sur la mesure de Borel associée à~$\lambda_\Sigma$,
considérée comme une mesure de Borel sur~$Y$.

Notons $\mathscr V^{\mathscr G,p}_{X_{\mathrm G},\pf}$
le sous-faisceau
de $\mathscr V^{\mathscr G,p}_{X_{\mathrm G}}$ 
constitué des formes-volume~$\lambda$
telles que $\lambda_\Bor$ soit localement finie.
Il existe un unique morphisme $\lambda\mapsto\lambda_\Rad$
de ce faisceau à valeurs dans le faisceau des mesures de Radon sur~$X$ qui,
pour tout domaine~$Y$ de~$X$, toute partie squelettique fermée~$\Sigma$ de~$Y$
et toute forme-volume~$\lambda_\Sigma$ sur~$\Sigma$
dont la mesure de Borel associée soit localement finie,
applique la classe de~$\lambda_\Sigma$ dans~$\mathscr V^{\mathscr G,p}(Y)$
sur la mesure de Radon associée à~$\lambda_\Sigma$,
considérée comme une mesure de Radon sur~$Y$.

Observons que l'on a 
$\mathscr V^{\mathscr G,p}_{X_\groth,\lf}=\mathscr V^{\mathscr G,p}_{X_\groth}$ si toute fonction de~$\mathscr G$ est localement bornée.
Si toute fonction de~$\mathscr G$ est continue,
on a aussi
on a $\lambda_\Bor=0$ si et seulement si $\lambda=0$,
et l'application $\lambda\mapsto \lambda_\Rad$ est 
une injection $\R$-linéaire.

\subsection{}\label{def.fonction-L1loc}
Soit $u\colon X\to\R\cup\{\pm\infty\}$ une fonction sur~$X$.
Les conditions suivantes sont équivalentes:
\begin{enumerate}\def\theenumi{\roman{enumi}}\def\labelenumi{(\theenumi)}
\item Pour toute partie squelettique~$\Sigma$ de~$X$ qui est une cellule
de dimension~$n$ et tout  isomorphisme paralinéaire~$f$ 
de~$\Sigma$ sur une cellule~$C$ de~$\R^n$, 
la fonction $u\circ f^{-1}$ est localement intégrable sur~$C$ pour la mesure 
de Lebesgue. 
\item Pour toute partie squelettique~$\Sigma$ de~$X$ qui est une cellule
de dimension~$n$, il existe un isomorphisme paralinéaire~$f$ 
de~$\Sigma$ sur une cellule~$C$ de~$\R^n$ tel que 
la fonction $u\circ f^{-1}$ soit localement intégrable sur~$C$ pour la mesure 
de Lebesgue. 
\item Pour toute forme volume~$\lambda$ sur~$X$
à coefficients localement bornés, la fonction~$u$
est localement intégrable par rapport à la mesure~$\lambda_\Bor$.
\end{enumerate}
On dit alors que $u$ est \emph{localement intégrable} sur~$X$.\index{fonction localement intégrable sur un espace G-tropical vertébré}
C'est une condition G-locale.

On note $\mathscr L^1_\loc(X)$ l'espace vectoriel
des fonctions localement intégrables,
à valeurs réelles, sur~$X$. 
C'est même un module sur l'anneau des fonctions numériques 
localement bornées sur~$X$.

Plus généralement, on notera $\mathscr L^1_{X,\loc}$ le G-faisceau
des fonctions localement intégrables sur~$X$ qui associe
à un domaine~$V$ de~$X$ l'ensemble des fonctions localement intégrables
sur~$V$.

\subsection{}
On dit qu'une partie~$N$ de~$X$ est \emph{négligeable}\index{partie négligeable!d'un espace G-tropical vertébré} si les propriétés 
équivalentes suivantes sont satisfaites:
\begin{enumerate}\def\theenumi{\roman{enumi}}\def\labelenumi{(\theenumi)}
\item Pour toute partie squelettique~$\Sigma$ de~$X$ qui est une cellule
de dimension~$n$ et tout  isomorphisme paralinéaire~$f$ 
de~$\Sigma$ sur une cellule~$C$ de~$\R^n$, 
le sous-espace~$f(N\cap\Sigma)$ de~$\R^n$ est négligeable pour la mesure 
de Lebesgue. 
\item Pour toute partie squelettique~$\Sigma$ de~$X$ qui est une cellule
de dimension~$n$, il existe un isomorphisme paralinéaire~$f$ 
de~$\Sigma$ sur une cellule~$C$ de~$\R^n$, 
le sous-espace~$f(N\cap\Sigma)$ de~$\R^n$ est négligeable pour la mesure 
de Lebesgue. 
\item La fonction caractéristique de~$N$ est localement intégrable
et pour toute forme volume~$\lambda$ sur~$X$
à coefficients localement bornés et à support compact, la fonction~$u$
est d'intégrale nulle par rapport à la mesure~$\lambda_\Bor$.
\end{enumerate}
C'est une propriété G-locale.


Toute partie $n$-négligeable d'une partie squelettique de~$X$ 
(déf.~\ref{PL-negligeable}) est négligeable.

Comme de coutume, on dira 
qu'une propriété sur~$X$ est vérifiée \emph{presque partout}
si l'ensemble des points de~$X$ où elle n'est pas vérifiée
est négligeable.
On notera ainsi $\mathrm L^1_\loc(X)$ l'espace vectoriel quotient 
de~$\mathscr L^1_\loc(X)$ par son sous-espace vectoriel 
des fonctions nulles presque partout.

Plus généralement, on notera $\mathrm L^1_{X,\loc}$ le G-faisceau
des fonctions localement intégrables sur~$X$ modulo celles nulles 
presque partout.

\clearpage

\section{Intégration des formes différentielles}

Soit $X$ un espace G-tropical calibré de niveau~$n$.

\subsection{}\label{sss.forme-volume}
Soit $\omega\in\mathscr F^{n,n}_{X_\groth}(X)$,
soit $\Sigma$ une partie squelettique fermée de~$X$ 
qui supporte fortement~$\omega$.
Par la construction du~\S\ref{sss.forme-volume-calibrage-pl},
en accouplant la restriction~$\omega|_\Sigma$
à~$\Sigma$ de la forme~$\omega$
et le $n$-calibrage~$\mu_\Sigma$ de~$\Sigma$, on obtient une forme-volume
$n$-dimensionnelle~$\tilde\omega_\Sigma$ sur~$\Sigma$, 
et donc une forme-volume $n$-dimensionnelle sur~$X$ 
au sens du~\S\ref{ss.forme-volume-Gtrop}.
D'après la proposition~\ref{prop.formes-volume-pl-f*} et la
définition d'un support fort, cette forme-volume sur~$X$
ne dépend pas du choix de~$\Sigma$;
on la note $\widetilde\omega$. 
Elle est positive si $\omega$ est une forme positive,
car le calibrage canonique est positif (lemme~\ref{lemm.calibrage-canonique}).

\begin{lemm}\label{lemm.forme-volume-restriction}
Soit $\omega\in\mathscr F^{n,n}_{X_\groth}(X)$. 
On suppose qu'il existe une partie squelettique fermée de~$X$
qui supporte fortement~$\omega$.
Pour tout domaine analytique~$Y$ de~$X$, il existe
une partie squelettique fermée de~$Y$ qui supporte fortement~$\omega|_Y$
et l'on a l'égalité $\widetilde{\omega|_Y}= \tilde\omega|_Y$
de formes-volume $n$-dimensionnelles sur~$Y$.
\end{lemm}
\begin{proof}
Soit $\Sigma$ une partie squelettique fermée de~$X$ qui supporte
fortement~$\omega$. Alors, $\Tau=\Sigma\cap Y$ est une partie
squelettique fermée de~$Y$ qui supporte fortement~$\omega|_Y$.
De plus, on a
\begin{align*}
 \widetilde{\omega|_Y} & = \langle (\omega|_Y)|_\Tau,\mu_\Tau \rangle
 && \text{par définition de~$\widetilde{\omega|_Y}$} \\
& = \langle \omega|_\Tau,\mu_\Tau\rangle && \text{d'après  (\ref{sss.forme-volume-restriction}) } \\
& = \langle (\omega|_\Sigma)|_\Tau,(\mu_\Sigma)|_\Tau\rangle && \text{d'après  (\ref{sss.forme-volume-restriction}) et la définition d'un calibrage} \\
& = \langle \omega|_\Sigma,\mu_\Sigma\rangle |_T
= \widetilde\omega|_Y 
\end{align*}
parce que l'accouplement entre $(n,n)$-formes et calibrages
est un morphisme de faisceaux sur le site~$X_\groth$.
Le lemme est ainsi démontré.
\end{proof}

\begin{prop}
Il existe un unique morphisme de faisceaux
$\mathscr F^{n,n}_{X_\groth}\to \mathscr V^n_{X_\groth}$
sur~$X_\groth$, noté $\omega\mapsto\widetilde\omega$,
qui prolonge l'application définie en~\textup{(\ref{sss.forme-volume})}
lorsque $\omega$
possède un support fort. 
Ce morphisme est $\mathscr F^{0,0}_{X_\groth}$-linéaire.
Si $\omega$  est une G-forme positive, 
la forme-volume $\widetilde\omega$ est positive.
Si $\omega$ est à coefficients dans une classe~$\mathscr G$,
il en va de même de la forme-volume~$\widetilde\omega$.
\end{prop}
\begin{proof}
Compte tenu du lemme~\ref{lemm.forme-volume-restriction},
la première assertion résulte de ce que toute forme sur un domaine 
paracompact possède un support fort, et que tout domaine 
est G-recouvert par ses domaines paracompacts.
Les autres propriétés découlent alors de la construction.
\end{proof}

\subsection{}
Explicitons cette construction lorsque $\omega$ est tropicale.
Soit $f\colon X\to E$ une carte G-tropicale et soit $\alpha$ une
$(n,n)$-forme sur un sous-espace paralinéaire de~$E$
contenant $f(X)$ telle que $\omega=f^*\alpha$.
Alors, $\Sigma_f^{(n)}$ est un support fort de~$\omega$
(lemme~\ref{lemm.Sigmaf-support-fort})
et on a  l'égalité de formes-volume sur~$\Sigma_f^{(n)}$ :
\begin{equation}
 \tilde\omega=\langle f_\trop^*\alpha,\mu_{\Sigma_f^{(n)}}\rangle. 
\end{equation}
Si, de plus, la carte G-tropicale~$f$ est agréable, 
la proposition~\ref{prop.formes-volume-pl-f*} entraîne alors que
\begin{equation}\label{n-forme-volume-but}
f_{\trop,*}(\tilde\omega)
 = \langle \alpha, f_*\mu_{\Sigma_f^{(n)}}\rangle
= \langle \alpha, \mu_{\Pi_f}\rangle.
\end{equation}

\subsection{}
\label{defi.integrale.n,n}
Soit $\omega$ une forme de type~$(n,n)$ sur~$X_\groth$
à coefficients tropicalement boréliens.
Nous venons de lui associer une 
forme-volume $n$-dimensionnelle~$\tilde\omega$.
Nous  noterons alors $\omega_{\Bor}$ la mesure de Borel sur~$X$
associée à~$\tilde\omega$ 
par la construction du~\S\ref{sss.mesure-forme-volume}.

Nous dirons que $\omega$ est
\emph{localement intégrable}\index{forme localement intégrable} si 
cette mesure est localement finie,
et que $\omega$ est intégrable si cette mesure est finie.
Lorsque $\omega$ est localement intégrable, nous noterons~$\omega_\Rad$
la mesure de Radon associée à~$\tilde\omega$;
on a $\abs{\omega_\Rad}=\omega_\Bor$. Lorsque $\omega$
est une G-forme positive, la  mesure~$\omega_\Rad$ est positive.

Si $f$ est une fonction borélienne sur~$X$,
on notera
$\int_X f \abs{\omega}$ et $\int_X f\omega$ 
les intégrales de~$f$
contre les mesures~$\omega_\Bor$ et~$\omega_\Rad$ lorsqu'elles existent.
Cela définit en particulier l'intégrale $\int_X\omega$
pour toute G-forme de type $(n,n)$ à support compact
qui est localement intégrable (par exemple lisse).

Les G-formes localement intégrables constituent un sous-espace vectoriel
de~$\mathscr F^{n,n}(X_\groth)$ contenant les G-formes lisses. 
Les G-formes intégrables en sont un sous-espace vectoriel qui
contient les G-formes lisses à support compact.

\subsection{}\label{ss.G-forme-L1loc}
Soit $u$ une fonction localement intégrable sur~$X$
et soit $\omega$ une G-forme lisse de type~$(n,n)$ sur~$X$.
Alors, $u$ est localement intégrable par rapport à la mesure~$\omega_\Bor$,
de sorte que $u\omega_\Rad$ est une mesure de Radon sur~$X$.

L'application $(u,\omega)\mapsto u\omega_\Rad$ est bilinéaire;
elle induit une application linéaire de 
$\mathscr L^1_\loc(X_\groth)\otimes \mathscr A^{n,n}(X_\groth)$
dans l'ensemble des mesures de Radon sur~$X$.
Cette application se factorise par $\mathrm L^1_\loc({X_\groth})
\otimes \mathscr A^{n,n}(X_\groth)$.
On la notera bien sûr
$\alpha\mapsto \alpha_\Rad$.

Cette construction s'étend en des morphismes de G-faisceaux
\[ \mathscr L^1_{\loc,X_\groth}\otimes_{\mathscr A_{X_\groth}} \mathscr A^{n,n}_{X_\groth}
\to
\mathscr L^1_{\loc,X_\groth}\otimes_{\mathscr A_{X_\groth}} \mathscr A^{n,n}_{X_\groth}
\to \mathscr M_{\Rad,X_\groth}
\]
sur~$X$, où $\mathscr M_{\Rad,X_\groth}$
est le G-faisceau des mesures de Radon sur~$X$.

Une section du faisceau
 $\mathscr L^1_{\loc,X_\groth}\otimes_{\mathscr A_{X_\groth}}\mathscr A^{p,q}_{X_\groth}$
sera appelée G-forme de type~$(p,q)$ à coefficients localement intégrables.

\section{Mesures de bord}

Soit $X$ un espace G-tropical calibré de niveau~$n$.

\subsection{}\label{sss.forme-volume-bord}
Soit 
$\omega\in\mathscr F^{n-1,n}_{X_\groth}(X)$,
soit $\Sigma$ une partie squelettique fermée de~$X$ 
qui supporte fortement~$\omega$.
En accouplant la restriction à~$\Sigma^{(n)}$ de la forme~$\omega$
et le calibrage canonique de~$\Sigma^{(n)}$ par
la construction du~\S\ref{sss.forme-bord-calibrage-pl},
on obtient une forme-volume
$(n-1)$-dimensionnelle sur~$\Sigma^{(n)}$. 
D'après la proposition~\ref{prop.formes-volume-bord-pl-f*}
et la définition d'un support fort,
la forme-volume $(n-1)$-dimensionnelle sur~$X$ qu'on en déduit
ne dépend pas du choix de~$\Sigma$; on la note~$\widetilde\omega$.


\begin{lemm}\label{lemm.forme-bord-restriction}
Soit $\omega\in\mathscr F^{n-1,n}_{X_\groth}(X)$. 
On suppose qu'il existe une partie squelettique fermée de~$X$
qui supporte fortement~$\omega$.
Pour tout domaine analytique \emph{ouvert}~$Y$ de~$X$, il existe
une partie squelettique fermée de~$Y$ qui supporte fortement~$\omega|_Y$
et l'on a l'égalité $\widetilde{\omega|_Y}= \tilde\omega|_Y$
de formes-volume $(n-1)$-dimensionnelles sur~$Y$.
\end{lemm}
\begin{proof}
Soit $\Sigma$ une partie squelettique fermée de~$X$ qui supporte
fortement~$\omega$. Alors, $\Tau=\Sigma\cap Y$ est une partie
squelettique fermée de~$Y$ qui supporte fortement~$\omega|_Y$.
De plus, on a
\begin{align*}
 \widetilde{\omega|_Y} & = \langle (\omega|_Y)|_\Tau,\mu_\Tau \rangle
 && \text{par définition de~$\widetilde{\omega|_Y}$} \\
& = \langle \omega|_\Tau,\mu_\Tau\rangle && \text{d'après  (\ref{sss.forme-volume-restriction}) } \\
& = \langle (\omega|_\Sigma)|_\Tau,(\mu_\Sigma)|_\Tau\rangle && \text{d'après  (\ref{sss.forme-volume-restriction}) et la définition d'un calibrage} \\
&= \langle \omega|_\Sigma,\mu_\Sigma\rangle|_\Tau = \widetilde\omega|_Y 
\end{align*}
car l'accouplement entre $(n-1,n)$-formes et calibrages
est un morphisme de faisceaux sur le site topologique de~$X$.
Le lemme est ainsi démontré.
\end{proof}

\begin{prop}
Notons $\pi\colon X_{\groth}\to X$ le morphisme canonique de sites.
Il existe un unique morphisme de faisceaux
$\pi_*\mathscr F^{n-1,n}_{X_\groth}\to \pi_*\mathscr V^{n-1}_{X_\groth}$
sur l'espace topologique~$X$, noté $\omega\mapsto\widetilde\omega$,
qui prolonge l'application définie en~\textup{(\ref{sss.forme-volume-bord})}
lorsque $\omega$ possède un support fort. 
Ce morphisme est $\mathscr F^{0,0}_X$-linéaire.
\end{prop}
\begin{proof}
Compte tenu du lemme~\ref{lemm.forme-bord-restriction},
cela résulte de ce que toute G-forme sur un ouvert 
paracompact possède un support fort, et que tout espace G-tropical
est recouvert par ses ouverts paracompacts.
\end{proof}

\subsection{}
Explicitons cette construction lorsque $\omega$ est tropicale.
Soit $f\colon X\to E$ une carte G-tropicale et soit $\alpha$ une
$(n-1,n)$-forme sur un sous-espace paralinéaire de~$E$
contenant $f(X)$ telle que $\omega=f^*\alpha$.
Alors, $\Sigma_f^{(n)}$ est un support fort de~$\omega$
(lemme~\ref{lemm.Sigmaf-support-fort})
et on a  l'égalité de formes-volume $(n-1)$-dimensionnelles
sur~$\Sigma_f^{(n)}$ :
\begin{equation}
 \tilde\omega=\langle f^*\alpha,\mu_{\Sigma_f^{(n)}}\rangle. 
\end{equation}
Si, de plus, l'application $f|_{\Sigma_f^{(n)}}$ est topologiquement
propre, 
la proposition~\ref{prop.formes-volume-pl-f*} entraîne alors que
\begin{equation}
\label{n-1-forme-volume-but}
f_{\trop,*}(\tilde\omega)
 = \langle \alpha, f_*\mu_{\Sigma_f^{(n)}}\rangle
= \langle \alpha, \mu_{\Pi_f}\rangle.
\end{equation}
Rappelons alors (\cf \S\ref{sss.forme-bord-pl-discordance})
que la forme-volume 
$(n-1)$-dimensionnelle~$\langle\alpha,\mu_{\Pi_f}\rangle$
se calcule par accouplement 
avec la discordance du calibrage de~$\Pi_f$.

\begin{prop}[Compensation des bords]
\label{prop.compensation}
Soit $(X_i)_{i\in I}$ un G-recouvrement localement fini de~$X$
et soit $\omega\in\mathscr F_{X_\groth}^{n-1,n}(X)$ une G-forme 
de type $(n-1,n)$ sur~$X$.
Pour toute partie~$J$ de~$I$, notons $X_J=\bigcap_{j\in J}X_j$
et $u_J$ l'immersion de~$X_J$ dans~$X$.
On a l'égalité
\[ \tilde\omega = \sum_{\emptyset \neq J\subset I}
 (-1)^{\Card(J)-1} (u_J)_* \widetilde{\omega|_{X_J}}. \]
\end{prop}
\begin{proof}
En raisonnant localement, on suppose que $X$ est paracompact
et donc que $\omega$ possède un support fort.
La formule découle alors de la proposition~\ref{prop.pl-compensation}.
\end{proof}

\subsection{}
\label{defi.integrale.n-1,n}
Soit $\omega$ une G-forme de type~$(n-1,n)$ sur~$X$ à coefficients
tropicalement boréliens.
Nous venons de lui associer une forme-volume 
$(n-1)$-dimensionnelle~$\tilde\omega$.
Nous  noterons alors $\omega_{\Bor}$ la mesure de Borel sur~$X$
associée par la construction du~\S\ref{sss.mesure-forme-volume}.

Lorsque la forme-volume~$\tilde\omega$ est localement intégrable,
la mesure~$\omega_\Bor$ est localement finie; on note alors~$\omega|_\Rad$
la mesure de Radon associée à~$\tilde\omega$.
On a $\abs{\omega_\Rad}=\omega_\Bor$.
 
Si $f$ est une fonction borélienne sur~$X$,
on notera
$\int_X^\partial f \abs{\omega}$ et $\int_X^\partial f\omega$ 
les intégrales de~$f$
contre les mesures~$\omega_\Bor$ et~$\omega_\Rad$ lorsqu'elles existent.
Cela définit en particulier l'intégrale $\int_X^\partial\omega$
pour toute G-forme de type $(n-1,n)$ à support compact
qui est localement intégrable (par exemple lisse).

\section{Intégrales et tropicalisations}

Soit $X$ un espace G-tropical calibré de niveau~$n$.

\begin{prop}\label{proposition.formule-au-but}
Soit $f\colon X\to E$ une carte G-tropicale agréable,
soit $\Pi_f$ le polyèdre tropical de la carte~$f$ 
et soit $\mu_{\Pi_f}$ son calibrage.

Soit $p\in\{n-1,n\}$ et soit $\alpha$ une $(p,n)$-forme à coefficients
boréliens sur un sous-espace
paralinéaire~$P$ de~$E$ contenant~$f(X)$.

\begin{enumerate}
\item
Si $p=n$, on a l'égalité
\[ f_{\trop,*} (f^*\alpha)_\Bor \geq \langle \alpha,\mu_{\Pi_f}\rangle_\Bor \]
de mesures de Borel sur~$\Pi_f^{(n)}$.

\item
Si $p=n-1$, on a l'inégalité
\[ f_{\trop,*} (f^*\alpha)_\Bor \geq \langle \alpha,\mu_{\Pi_f}\rangle_\Bor \]
de mesures de Borel sur~$\Pi_f^{(n)}$.

\item
La mesure $(f^*\alpha)_\Bor$ est localement finie 
si et seulement si la forme $\alpha|_{\Pi_f}$ est localement intégrable
sur~$\Pi_f$;  
dans ce cas, on a l'égalité 
\[ f_{\trop,*} (f^*\alpha)_\Rad = \langle \alpha,\mu_{\Pi_f}\rangle_\Rad \]
de mesures de Radon sur~$\Pi_f$.

\item
En particulier, si $\alpha$ est intégrable sur~$\Pi_f$,
la forme $f^*\alpha$ est intégrable sur~$X$,
on  a l'égalité
\begin{equation} \label{eq.integrale-but}
 \int_X  f^*\alpha = \int_{\Pi_f} \langle \alpha,\mu_{\Pi_f}\rangle  \end{equation}
lorsque $p=n$ et l'égalité
\begin{equation} \label{eq.integrale-but-bord}
 \int_X^\partial  f^*\alpha = \int_{\Pi_f}^\partial \langle \alpha,\mu_{\Pi_f}\rangle  \end{equation}
lorsque $p=n-1$.
\end{enumerate}
\end{prop}
\begin{proof}
Posons $\omega=f^*\alpha$. 
D'après la relation~\eqref{n-forme-volume-but} lorsque $p=n$
et la relation~\eqref{n-1-forme-volume-but} lorsque $p=n-1$, 
on a l'égalité $(f_\trop)_*(\tilde\omega)=\langle\alpha,\mu_{\Pi_f}\rangle$
de formes-volume $p$-dimensionnelles sur $\Pi_f$.
Le lemme~\ref{lemm.f_*-Bor-Rad} entraîne alors
que $f_*(\tilde\omega_\Bor)\geq \langle\alpha,\mu_{\Pi_f}\rangle_\Bor$.
Loreque $p=n$, on vérifie, en revenant aux définitions
et en utilisant la positivité du calibrage canonique,
que l'on a en fait l'égalité, d'où~(1).
En appliquant le lemme~\ref{lemm.f_*-Bor-Rad} aux mesures
de Radon correspondantes, on en déduit~(2), puis~(3) par intégration.
\end{proof}

\subsection{}\label{ss.integrale-but-suffit}
La formule~\eqref{eq.integrale-but} définit explicitement
l'intégrale d'une forme tropicale. Prouvons inversement que cette relation
dans des cartes tropicales permet de reconstituer
la mesure de Radon associée à une G-forme de type~$(n,n)$
et en donne ainsi une définition tropicale. 

Soit $\omega$ une G-forme
de type~$(n,n)$ sur~$X$; supposons-la à support compact et 
à coefficients tropicalement boréliens localement bornés.

Soit $(V_i)_{i\in I}$ une famille finie de domaines compacts
de~$X$ tropicalisant~$\omega$ 
et dont dont la réunion~$V$ recouvre le support de~$\omega$.
Pour toute partie non vide~$J$ de~$I$, posons $V_J=\bigcap_{j\in J} V_j$.
On a alors
\[ \int_X\omega=\int_V\omega 
= \sum_{\emptyset\neq J\subset I} (-1)^{\Card(J)-1}  \int_{V_J}\omega. \]
Pour toute partie non vide~$J$, la forme~$\omega$ est tropicale sur~$V_J$,
et $V_J$ est compact, si bien que toute carte G-tropicale sur~$V_J$.
Ainsi, l'intégrale $\int_{V_J}\omega$
peut se calculer à l'aide de la formule~\eqref{eq.integrale-but}.

\section{Formules de Stokes et Green}

Soit $X$ un espace G-tropical calibré de niveau~$n$.

\begin{prop}\label{prop.stokes}
\begin{enumerate}
\item
Soit $\omega$ une G-forme à coefficients~$\mathscr C^1$
et à support compact sur~$X$,
de type~$(n-1,n)$. On a la \emph{formule de Stokes} :
\[ \int_X \di\omega = \int^\partial_X \omega. \]
\item
Soit $\alpha$ et $\beta$ des G-formes à coefficients~$\mathscr C^2$
sur~$X$,  symétriques,
de type~$(p,p)$ et~$(q,q)$ respectivement, où $p+q=n-1$.
Si l'intersection de leurs supports est compact, on a
la \emph{formule de Green}:
\[ \int_X (\alpha\wedge \ddc \beta - \ddc\alpha\wedge\beta) 
= \int^\partial_X (\alpha\wedge\dc\beta-\dc\alpha\wedge\beta). \]
\end{enumerate}
\end{prop}
\begin{proof}
Comme le support de~$\omega$ est compact, il existe
une partie squelettique fermée~$\Sigma$ de~$X$ qui la supporte fortement.
La formule de Stokes découle alors de la formule correspondante
en géométrie paralinéaire (proposition~\ref{prop.stokes-pl}),
appliquée à~$\Sigma^{(n)}$, à la forme~$\omega|_{\Sigma^{(n)}}$ 
et au calibrage canonique de~$\Sigma^{(n)}$.

La formule de Green s'en déduit comme dans le lemme~\ref{lemm.green-affine}
en appliquant la formule de Stokes à la forme à support compact
$\omega= \alpha\wedge\dc\beta-\dc\alpha\wedge\beta$.
Lorsque $\alpha$ et~$\beta$ sont toutes deux à support compact,
elle découle aussi de la formule de Green en géométrie paralinéaire,
appliqué à une partie polyédrale qui supporte
fortement
ces deux formes.
\end{proof}


\chapter{Espaces tropicaux}
\section{Espaces G-tropicaux affinés}
\label{sec.G-trop-aff}

Soit $X$ un espace G-tropical.

\subsection{}
On se donne un sous-préfaisceau en $\Q$-espaces
vectoriels~$\mathrm L_X$ du G-faisceau $\PL_X$ qui contient les fonctions
constantes,
qu'on appelle le préfaisceau des fonctions affines sur~$X$.

On suppose:
\begin{enumerate}
\item Pour tout domaine~$V$ de~$X$, le
préfaisceau sur~$V$ induit par~$\mathrm L_X$ est un faisceau
pour sa topologie usuelle.
Autrement dit, \emph{toute fonction paralinéaire~$f$ sur un domaine~$V$ 
de~$X$ qui est localement affine est affine.}
\item Pour tout domaine~$V$ de~$X$,
une fonction sur~$V$ est paralinéaire si et seulement si 
elle est G-localement affine.
\end{enumerate}

On dit qu'une carte G-tropicale $(V,f)$ sur un domaine
est \emph{tropicale}  si $f$ est affine.

\subsection{}
Un espace G-tropical muni d'un tel sous-préfaisceau
est appelé \emph{espace G-tropical affiné}.\index{espace G-tropical!affiné}

\subsection{}
Tout domaine d'un espace G-tropical affiné
est naturellement muni d'une structure d'espace G-tropical affiné.

\begin{exem}
Soit $X$ un sous-espace paralinéaire d'un espace affine~$E$.
Soit $\mathrm L_X$ le G-préfaisceau 
dont les sections sur un sous-espace paralinéaire~$U$
de~$X$ sont les fonctions sur~$U$ qui sont localement (pour la topologie
ordinaire) restrictions de fonctions affines sur~$E$.
Alors $\mathrm L_X$ est un affinage de~$X$
qui munit~$X$ d'une structure d'espace paralinéaire affiné.
\end{exem}

\begin{exem}
Nous définirons plus tard l'espace G-tropical associé  
à un espace analytique au sens de Berkovich
et démontrerons (proposition~\ref{prop.parties-polyedrales-moments})
que 
la restriction des fonctions de la forme $r \log (\abs{f})+c$,
où $f$ est une fonction analytique inversible, $r\in\Q$ et $c\in\R$,
en constituent un affinage.
\end{exem}

\subsection{}
Si $X$ et $Y$ sont des espaces G-tropicaux affinés,
un morphisme $f\colon X \to Y$ d'espaces G-tropicaux
est dit affine  si pour tout domaine~$V$ de~$Y$
et toute fonction affine $\phi\in\mathscr L_Y(V)$,
la fonction $\phi\circ f$ sur le domaine~$f^{-1}(V)$  de~$X$
est affine.

Si $f\colon X\to E $ est une application paralinéaire
vers un espace affine~$E$ 
muni de son affinage canonique,
il revient au même de dire que $f$ est un morphisme
affine ou que les composantes de~$f$ (dans un repère affine donné)
sont des fonctions affines sur~$X$.

\section{Formes lisses sur un espace G-tropical affiné}
Soit $\mathscr G$ une classe  de coefficients.

\begin{defi}
Soit $(p,q)$ un couple d'entiers naturels.
Soit $\omega$ 
une G-forme de type~$(p,q)$ à coefficients dans~$\mathscr G$
sur un domaine~$V$ de~$X$.

On dit que $\omega$ est \emph{tropicale}
s'il existe une carte tropicale qui G-tropicalise~$\omega$.

On dit que $\omega$ est une \emph{forme}  
(de type~$(p,q)$ à coefficients dans~$\mathscr G$)
si pour tout point~$x$ de~$V$,
il existe un \emph{voisinage ouvert}~$U$  de~$x$ dans~$V$
sur lequel $\omega$ est tropicale.
\end{defi}

Lorsque $\mathscr G$ est la classe des fonctions lisses,
on parle de forme lisse
sur l'espace G-tropical affiné~$X$.\index{forme lisse sur un espace G-tropical affiné}
Si de plus $p=q=0$, on parle de fonction lisse.

\subsection{}
Les formes lisses forment un sous-G-préfaisceau 
en $\R$-espaces vectoriels de~$\mathscr A^{p,q}_X$
qui est un faisceau pour la topologie usuelle.
Ces faisceaux sont stables par les opérateurs~$\di$, $\dc$, $\mathrm J$
et par le produit extérieur.

\subsection{}
Un espace G-tropical affiné~$(X,\mathrm L_X)$
est dit affinement séparé
si pour tout couple $(x,y)$ de points de~$X$ avec $x\neq y$,
il existe une fonction affine~$f$ sur~$X$ telle que $f(x)\neq f(y)$.
A fortiori, les fonctions lisses sur~$X$ séparent les points.
On dira alors que $X$ est différentiellement séparé.

Si $X$ est affinement (resp. différentiellement)
séparé, il en est de même de tout domaine.

Un espace G-tropical affiné $(X,\mathrm L_X)$ est dit 
localement affinement (resp. différentiellement) séparé 
si chacun de ses points possède 
un voisinage affinement (resp. différentiellement) séparé.

Cela entraîne que le faisceau des fonctions lisses sépare localement
les points au sens du~\S\ref{ss.app-sep}.

\begin{coro}\label{coro.trop-lisse-separe}
Soit $(X,\mathrm L_X)$ un espace G-tropical affiné
qui est localement différentiellement séparé.
Soit $K$ une partie compacte de~$X$ et soit $U$ un voisinage ouvert de~$K$.
Il existe une fonction lisse sur~$X$ qui vaut~$1$ au voisinage de~$K$
et dont le support est compact et contenu dans~$U$.
\end{coro}
\begin{proof}
C'est une conséquence du corollaire~\ref{coro.app-lisse-separe}.
\end{proof}

\begin{theo}[Stone--Weierstraß] \label{theo.trop-sw}
Soit $(X,\mathrm L_X)$ un espace G-tropical affiné qui est localement
différentiellement séparé et soit $U$ un ouvert de~$X$.
Soit $f$ une fonction continue à valeurs réelles sur~$X$
dont le support est compact et contenu dans~$U$.
Soit $a$ et $b$ des nombres réels tels que $a\leq f\leq b$.
Pour tout~$\eps>0$, il existe une fonction lisse~$g$ sur~$X$
dont le support est compact, contenu dans~$U$,
telle que $\abs{f(x)-g(x)}<\eps$ pour tout $x\in X$
et telle que $a\leq g\leq b$.
\end{theo}
\begin{proof}
C'est une conséquence du théorème~\ref{theo.app-sw}.
\end{proof}

\begin{coro}\label{coro.trop-sw}
Soit $(X,\mathrm L_X)$ un espace G-tropical affiné
qui est localement différentiellement séparé.

\begin{enumerate}
\item L'ensemble des fonctions lisses à support compact sur~$X$
est dense dans l'espace~$\mathscr C_\cpct(X)$
des fonctions continues à support compact sur~$X$, 
muni de la topologie limite inductive des espaces
de fonctions à support dans un compact donné.
\item L'ensemble des fonctions lisses à support compact sur~$X$
est dense dans l'espace~$\mathscr C(X)$ des fonctions continues sur~$X$
muni de la topologie de la convergence uniforme sur tout compact.
\end{enumerate}
\end{coro}
\begin{proof}
C'est une conséquence du corollaire~\ref{coro.app-sw}.
\end{proof}

\begin{prop}[Partitions de l'unité lisses]
\label{prop.trop-partition.1.lisse}
On suppose que $X$ est paracompact et
que $(X, \mathrm L_X)$ est localement
différentiellement séparé.
Soit $(U_i)_{i\in I}$ un recouvrement ouvert de~$X$.
Il existe une partition de l'unité\index{partition de l'unité} 
subordonnée à~$(U_i)$
formée de fonctions lisses sur~$X$.
\end{prop}
\begin{proof}
Cela résulte de la proposition~\ref{prop.app-partition.1.lisse}.
\end{proof}

\begin{coro}
\label{prop.trop-partition.1.lisse-compact}
Soit $F$ une partie fermée de~$X$ qui est dénombrable à l'infini
(par exemple, compacte).
On suppose que $(X, \mathrm L_X)$ est localement différentiellement séparé.
Soit $(U_i)_{i\in I}$ un recouvrement de~$F$ par des ouverts de~$X$.
Il existe une famille localement finie $(\theta_i)$ 
de fonctions lisses sur~$X$, à support compact, positives,
telle que pour tout~$i$, le support de~$\theta_i$ soit contenu dans~$U_i$
et que $\sum\theta_i$ vaille~$1$ au voisinage de~$F$.
\end{coro}
\begin{proof}
Comme $F$ est dénombrable à l'infini, elle possède un voisinage ouvert~$U$
dans~$X$ qui est paracompact
(lemme~\ref{comp.fort.paracomp}). 
Comme $(U, \mathrm L_U)$ 
est localement différentiellement séparé,
l'assertion résulte 
de la proposition~\ref{prop.app-trop-partition.1.lisse-compact}.
\end{proof}

\begin{coro}\label{coro.pl-fin}
On suppose que $X$ est paracompact et que $(X,\mathrm L_X)$
est localement différentiellement séparé.
Les faisceaux $\mathscr A^{p,q}_X$ de  formes lisses sur~$X$ sont fins.
\end{coro}
\begin{proof}
D'après le corollaire~\ref{coro.app-fin},
le faisceau des fonctions lisses sur~$X$ est fin.
Comme les faisceaux de formes sont des faisceaux de modules
sur le faisceau des fonctions lisses, ils sont fins.
\end{proof}

\subsection{}
Décrivons maintenant la topologie naturelle sur les espaces des formes;
nous reprenons les notations introduites au paragraphe~\ref{ss.topologie-jeu-de-carte}
lorsque nous avons définie la topologie des espaces de G-formes.

Soit $W$ un domaine compact de~$X$.
Notons $\mathscr D^\circ_W$ l'ensemble des jeux de cartes~$D$ 
sur~$X$ tels que la réunion des intérieurs des domaines de cartes
appartenant à~$D$ recouvre~$W$.
L'espace $\mathscr A^{p,q}(X)$ est réunion croissante, lorsque~$D$
parcourt l'ensemble préordonné filtrant~$\mathscr D^\circ_W(X)$
des sous-espaces $\mathscr A^{p,q}_D(X)$.

\section{Description locale d'objets G-locaux}

Nous introduisons maintenant la  notion d'espace G-tropical affiné décent:
ce sont des espaces possédant suffisamment de fonctions affines pour
pouvoir décrire localement les fonctions paralinéaires
et les domaines. L'intérêt de cette notion est qu'alors
les squelettes et les G-formes peuvent alors être décrites \emph{localement}
au moyen de fonctions affines.

\begin{defi}\label{defi.espace-decent}
On dit qu'un espace G-tropical affiné~$X$ est \emph{décent} 
en un point~$x$ si 
pour tout domaine~$V$ de~$X$ contenant~$x$,
les deux conditions suivantes sont vérifiées:
\begin{enumerate}
\item
Pour toute fonction affine~$f$ sur~$V$,
il existe un voisinage ouvert~$U$ de~$x$ dans~$X$
et une fonction affine~$g$ sur~$U$ 
telle que $f$ et~$g$ coïncident sur $U\cap V$;
\item
Il existe un voisinage ouvert~$U$ de~$x$ dans~$X$
tel que $U\cap V$ soit défini dans~$U$ par conjonction
et disjonction d'inégalités larges entre fonctions affines sur~$U$.
\end{enumerate}
\index{espace G-tropical affiné décent}
\end{defi}

On dit que $X$ est décent s'il est décent en chacun de ses points.

Tout domaine d'un espace G-tropical affiné décent est décent.

\begin{prop}\label{prop.decent-extension}
Soit $X$ un espace G-tropical affiné et soit $\mathscr G$
une classe de coefficients.
Soit $V$ un domaine de~$X$ et soit $\omega$ une forme sur~$V$
à coefficients dans~$\mathscr G$.

\begin{enumerate}
\item Soit $x$ un point de~$V$ en lequel $X$ est décent.
Il existe un voisinage domanial~$U$ de~$x$ dans~$X$ et une forme~$\omega'$
de classe~$\mathscr G$ sur~$U$ 
telle que $\omega$ et $\omega'$ coïncident sur~$V\cap U$.
\item Si le support de~$\omega$ est décent et dénombrable à l'infini
et si $X$ est localement différentiellement séparé, il existe
une forme~$\omega'$ de classe~$\mathscr G$ sur~$X$ qui prolonge~$\omega$.
\item Si le support de~$\omega$ est décent et compact 
et si $X$ est localement différentiellement séparé, 
il existe une forme~$\omega'$ de classe~$\mathscr G$ à support compact sur~$X$ 
qui prolonge~$\omega$.
\end{enumerate}
\end{prop}
\begin{proof}
\begin{enumerate}
\item Soit $V'$ un voisinage domanial compact de~$x$ dans~$V$
sur lequel $\omega$ est tropicale,
soit $f\colon V'\to\R^n$ une carte tropicale et soit $\alpha$
une forme de classe~$\mathscr G$ sur $f(V')$ telle que $\omega|_{V'}=f^*\alpha$.

La partie~$f(V')$ de~$\R^n$ est compacte.
D'après le lemme~\ref{lemm.extensions},
il existe une forme~$\alpha'$ de classe~$\mathscr G$ 
sur~$\R^n$ qui prolonge~$\alpha$.

Comme $X$ est décent en~$x$ (définition~\ref{defi.espace-decent}), 
il existe 
une fonction affine~$f'$ 
définie sur un voisinage ouvert~$U$ de~$x$ qui prolonge~$f|_{U\cap V'}$.
Quitte à restreindre~$U$, on peut supposer que $U\cap V$ est contenu
dans le voisinage~$V'$ de~$x$ dans~$V$; alors, $U\cap V'=U\cap V$. 
Posons $\omega=f'^*\alpha'$; c'est une forme de classe~$\mathscr G$
sur~$U$ qui prolonge la restriction de~$\omega$ à~$U\cap V$.

\item 
Pour tout $x\in \supp(\omega)$, considérons
un voisinage ouvert~$U_x$ de~$x$ et une forme~$\omega_x$ 
sur~$U_x$, de classe~$\mathscr G$, qui prolonge $\omega|_{U_x\cap V}$.
Soit $(\theta_x)$ une famille localement finie de fonctions lisses sur~$X$,
à support compact, positives, telle que pour tout~$x$,
le support de~$\theta_x$ soit contenu dans~$U_x$
et que $\sum \theta_x$ vaille~$1$ au voisinage de~$\supp(\omega)$
(proposition~\ref{prop.trop-partition.1.lisse-compact}).
Soit $\omega'$ la somme localement finie $\sum_x \theta_x \omega_x$;
c'est une forme de classe sur~$\mathscr G$ sur~$X$ 
qui prolonge~$\omega$.

\item
Comme le support de~$\omega$ est compact, il est dénombrable à l'infini,
et la forme construite dans l'alinéa précédent est à support compact.
\qedhere
\end{enumerate}
\end{proof}

\begin{prop}\label{prop.decent-extensible}
Soit $X$ un espace G-tropical vertébré de niveau~$n$, affiné, très riche.
Soit $\Sigma$ une partie squelettique de~$X$
et soit $x$ un point de~$\Sigma$ en lequel $X$ est décent.
Il existe un voisinage ouvert~$U$ de~$x$ dans~$X$
et une carte tropicale~$f$ sur~$U$ telle que 
$\Sigma\cap U\subset \Sigma_f^{(n)}$.
\end{prop}
\begin{proof}
Soit $V$ un voisinage domanial compact de~$x$.
Comme $X$ est très riche, il existe une partie squelettique compacte~$\Sigma'$
de~$X$, purement de dimension~$n$, telle que $\Sigma\cap V \subset \Sigma'$.
Cela permet de supposer que $\Sigma$ est compact
et purement de dimension~$n$.

D'après la proposition~\ref{prop.pur-n-squelette},
il existe une famille finie $(V_i,f_i)$ de  cartes G-tropicales
définies sur des domaines compacts, 
telle que $\Sigma\subset \bigcup_i \Sigma_{f_i}^{(n)}$.
Si l'on restreint la réunion précédente aux indices~$i$
tels que $x\in V_i$, on obtient un voisinage de~$x$ dans~$\Sigma$.
On peut donc supposer que $x\in V_i$ pour tout~$i$.

L'espace~$X$ est décent.
On peut donc choisir, pour tout~$i$, un voisinage~$U_i$ de~$x$
et une carte tropicale $g_i$ sur~$U_i$
qui coïncide avec~$f_i$ sur $V_i\cap U_i$.
Soit $U$ l'intersection des~$U_i$
et soit $g$ la carte tropicale sur~$U$ obtenue par
concaténation des~$g_i$.
Pour tout~$i$, on a $\Sigma_i\cap U\subset \Sigma_{g_i}^{(n)}\subset
\Sigma_g^{(n)}$. Par suite, $\Sigma_g^{(n)}$ contient~$\Sigma\cap U$.
\end{proof}

\begin{prop}\label{prop.local-G-local}
Soit $X$ un espace G-tropical affiné, soit $\omega$ une G-forme
de classe~$\mathscr G$ sur~$X$ et soit $x$ un point de~$X$
en lequel $X$ est décent.
Il existe un voisinage domanial compact~$V$ de~$x$ dans~$X$,
une carte tropicale $f\colon V\to E$ et une G-forme~$\alpha$
de classe~$\mathscr G$ sur $f(V)$ telles que $\omega|_V=f^*\alpha$.
\end{prop}
\begin{proof}
Soit $(V_i)$ une famille finie de domaines compacts 
dont la réunion~$V$ est un voisinage de~$x$ et qui tropicalisent
la G-forme~$\omega$.
Pour tout~$i$, on fixe un carte tropicale~$f_i$ sur~$V_i$ 
qui tropicalise~$\omega|_{V_i}$ et une forme $\alpha_i$
sur $f_{i}(V_i)$ telles que $\omega|_{V_i}=f_i^*\alpha_i$.

Comme $X$ est décent en~$x$, il existe des cartes tropicales~$g_i$
définies au voisinage de~$x$ telles que $g_i$ et~$f_i$
coïncident sur~$V_i$ au voisinage de~$x$.
Quitte à restreindre les~$V_i$ (et donc leur réunion~$V$) 
et à remplacer~$f_i$ par~$g_i$,
on peut donc supposer que chaque~$f_i$ est définie sur~$V$.

Pour chaque~$i$, il existe une famille $(g_{i,j})$ de fonctions 
affines au voisinage de~$x$ 
telle  que $V_i$ soit défini sur
un voisinage compact~$W$ de~$x$ dans~$V$
par conjonction et disjonction
d'inégalités de la forme $g_{i,j}\leq g_{i,j}(x)$.

Considérons la carte tropicale~$h$ sur~$W$ définie par la réunion
des familles $(f_i)$ et $(g_{i,j})$.
Par construction, on a $W\cap V_i= h^{-1}(h(W\cap V_i))$;
en effet, si $y\in W$ et $z\in W\cap V_i$  vérifient
$h(y)=h(z)$, alors $g_{i,j}(y)=g_{i,j}(z)$
pour tout~$j$, si bien que $y$ vérifie le système d'inégalités
qui définit~$V_i$ dans~$W$; autrement dit, $y\in V_i$.
Par conséquent,  on a aussi
$h(W\cap V_i\cap V_j) = h(W\cap V_i) \cap h(W\cap V_j)$. 

On déduit de~$\alpha_i$ une forme~$\beta_i$ sur  $h(W\cap V_i)$
telle que $h|_{W\cap V_i}^* \beta_i=\omega|_{W\cap V_i}$.
Sur $W\cap V_i\cap V_j$, les formes $h|_{W\cap V_i}^*\beta_i $ et $h|_{W\cap V_j}^*\beta_j$ coïncident toutes deux avec $\omega|_{W\cap V_i\cap V_j}$.
Par suite, $\beta_i$ et $\beta_j$ coïncident sur $h( W\cap V_i\cap V_j)$
qui est égal à $h(W\cap V_i) \cap h(W\cap V_j)$ comme on l'a vu.
Ces formes~$\beta_i$ se recollent en une G-forme~$\beta$ sur~$h(W)$.
Pour tout~$i$, on a $h^*\beta|_{W\cap V_i}=h|_{W\cap V_i}^*\beta_i
=\omega|_{W\cap V_i}$. Il s'ensuit $h^*\beta = \omega$.
\end{proof}

\begin{coro}
Soit $X$ un espace G-tropical affiné qui est décent en tout point.
Une fonction G-localement tropicalement continue sur~$X$
(resp.  tropicalement borélienne)  est tropicalement 
continue (resp. borélienne).
En particulier, la fonction indicatrice d'un domaine fermé
de~$Y$ est tropicalement borélienne.
\end{coro}
\begin{proof}
En effet, toute fonction G-continue (resp. borélienne)
sur un sous-espace paralinéaire de~$\R^n$ est continue
(resp. borélienne).
La seconde assertion résulte de l'exemple~\ref{exem.indicatrice-borel}.
\end{proof}

\section{Topologies sur les espaces de formes lisses}
\label{sec.topologie-formes-lisses}

\subsection{}\label{subsec.topologie-formes-lisses}
Soit $X$ espace G-tropical affiné.
Disons qu'un jeu de \emph{cartes tropicales} sur~$X$
est couvrant si la réunion  des intérieurs de ses domaines
de définition recouvre~$X$.  Un tel jeu de cartes est G-couvrant.

Notons $\Gamma_X$ la famille des jeux de cartes tropicales qui
sont couvrants. Elle est filtrante.
Par construction, l'espace des $(p,q)$-formes sur~$X$
coïncide avec l'espace $\mathscr A^{p,q}_{\Gamma_X}(X)$.
Il hérite ainsi d'une topologie~$\mathscr C^m$,
pour tout $m\in\N\cup\{+\infty\}$.
Cette topologie est moins fine que la topologie induite
par celle des G-formes ; plus explicitement, 
l'inclusion des formes dans les G-formes est continue
pour les topologies~$\mathscr C^m$.

De même, pour toute famille de supports~$\Phi$ sur~$X$,
on en déduit des topologies~$\mathscr C^m$
sur l'espace $\mathscr A^{p,q}_{\Phi}(X)$ des formes à
support dans~$\Phi$ sur~$X$, moins fines
que les topologies induites par celles des G-formes correspondantes.

\begin{prop}\label{prop.continuite-wedge}
Soit $\alpha$ une forme de type~$(p',q')$ sur~$X$.
Soit $\Phi$ et $\Psi$ des familles de supports sur~$X$;
on suppose que pour toute partie fermée~$F$ de~$\supp(\alpha)$
appartenant à~$\Phi$ appartient également à~$\Psi$.

Pour toute forme $\omega\in\mathscr A^{p,q}_{\Phi}(X)$,
la forme $\alpha\wedge\omega$ appartient à $\mathscr A^{p,q}_{\Psi}(X)$
et l'application 
\[   A^{p,q}_{\Phi}(X) \to \mathscr A^{p+p',q+q'}_{\Psi}(X), 
\qquad \omega\mapsto \alpha\wedge\omega \]
est linéaire continue.
\end{prop}
\begin{proof}
Cela découle de la proposition~\ref{prop.continuite-wedge}
appliquée à la famille filtrante~$\Gamma_X$.
\end{proof}

\begin{prop}\label{prop.continuite-pull-back}
Soit $u\colon Y\to X$ un morphisme d'espaces G-tropicaux affinés.
Soit $\Phi$ une famille de supports sur~$X$ et 
soit~$\Psi$ une famille de supports sur~$Y$;
on suppose que pour tout fermé $F$ de $X$ appartenant
à~$\Phi$, la partie fermée $u^{-1}(F)$ de~$Y$ appartient à~$\Psi$.

Soit $(p,q)$ un couple d'entiers naturels et 
soit $m\in\N\cup\{\infty\}$.

Pour toute forme $\alpha\in\mathscr A^{p,q}_{\Phi}(X)$,
la G-forme $u^*\alpha$ est une forme appartenant à~$\mathscr A^{p,q}_{\Psi}(Y)$.
L'application
\[ u^*\colon \mathscr A^{p,q}_{\Phi}(X) \to
\mathscr A^{p,q}_{\Psi}(Y), \qquad
\alpha\mapsto u^*\alpha \]
est linéaire et continue.
\end{prop}
\begin{proof}
Cela découle de la proposition~\ref{prop.continuite-pull-back}
appliquée aux familles filtrantes~$\Gamma_X$ et~$\Gamma_Y$,
une fois remarqué que l'on $\Gamma_Y\preceq_u\Gamma_X$.
\end{proof}

\begin{coro}\label{coro.continuite-pull-back'}
Soit $f\colon Y\to X$ un morphisme d'espaces G-tropicaux affinés.
Soit $m\in\N\cup\{+\infty\}$.

\begin{enumerate}
\item L'application $f^*\colon \mathscr A^{p,q}(X)
           \to\mathscr A^{p,q}(Y)$ est continue
pour les topologies~$\mathscr C^m$ sur ces espaces;
\item Si le morphisme~$f$ est \emph{compact},
l'application $f^*\colon \mathscr A^{p,q}_{\cpct }(X)
           \to\mathscr A^{p,q}_{\cpct }(Y)$ est continue
pour les topologies~$\mathscr C^m$ sur ces espaces.
\end{enumerate}
\end{coro}

\begin{prop}\label{prop.continuite-push-U'}
Soit $X$ un espace G-tropical affiné et soit $U$ une partie ouverte de~$X$;
notons $j\colon U\to X$ l'inclusion.
Soit $\Phi$ une famille de supports sur~$U$,
soit $\Psi$ une famille de supports sur~$X$; 
on suppose que pour toute partie $F\in\Phi$, 
la partie~$F$ est fermée dans~$X$ et appartient à~$\Psi$.
(En particulier, on peut prendre pour $\Phi$ et~$\Psi$
les ensembles de parties compactes de~$U$ et~$X$ respectivement.)

\begin{enumerate}
\item Pour toute forme~$\omega\in\mathscr A^{p,q}_{\Phi}(U)$,
la G-forme $j_!\omega$ sur~$X$ est une forme qui appartient 
à $\mathscr A^{p,q}_{\Psi}(X)$.
et l'application 
\[ j_!\colon\mathscr A^{p,q}_{\Phi }(U)\to \mathscr A^{p,q}_{\Psi }(X) \]
ainsi définie 
est linéaire et continue pour les topologies~$\mathscr C^m$.

\item
Cette application est injective.

\item
Supposons que $\Gamma\preceq_j \Delta$ et que toute
partie fermée~$F$ de~$X$ appartenant à~$\Psi$ qui contenue dans~$U$
appartient à~$\Phi$.
Supposons que plus que $d_\trop(x)<\max(p,q)$
pour tout $x\in X\setminus U$.
Alors cette application~$j_!$ est un homéomorphisme, d'inverse donné par~$j^*$.
\end{enumerate}
\end{prop}
\begin{proof}
Cela résulte de la proposition~\ref{prop.continuite-push-U},
appliqué aux familles~$\Gamma_U$ et~$\Gamma_X$,
ce qui est licite car pour tout jeu de cartes tropicales
couvrant~$U$ et toute partie~$F$ de~$U$ qui est fermée dans~$X$,
le jeu de carte obtenu en adjoignant la carte
tropicale constante de domaine~$X\setminus U$ est couvrant sur~$X$.
\end{proof}

\begin{prop}\label{prop.continuite-di-dc-J'}
Soit $X$ un espace G-tropical affiné,
soit $\Phi$ une famille de supports sur~$X$, 
et soit $m\in\N\cup\{+\infty\}$.
\begin{enumerate}
\item Pour toute forme $\omega\in\mathscr A^{p,q}_{\Phi}(X)$,
la forme $\mathrm J\omega$ appartient à~$\mathscr A^{q,p}_{\Phi}(X)$. 
L'application~$\mathrm J$ ainsi définie est linéaire et continue.
\item Supposons $m\geq 1$.
Pour toute G-forme $\omega\in\mathscr A^{p,q}_{\Phi}(X)$,
on a $\di\omega \in\mathscr A^{p+1,q}_{\Phi}(X)$
et $\dc\omega \in\mathscr A^{p,q+1}_{\Phi}(X)$.
Les applications~$\di$ et~$\dc$ ainsi définies sont continues
lorsqu'on munit $\mathscr A^{p,q}_{\Phi}(X)$
de la topologie~$\mathscr C^m$ et les
espaces $\mathscr A^{p+1,q}_{\Phi}(X)$
et $\mathscr A^{p,q+1}_{\Phi}(X)$
de la topologie~$\mathscr C^{m-1}$.
\end{enumerate}
\end{prop}
\begin{proof}
L'assertion découle de la proposition~\ref{prop.continuite-di-dc-J},
appliquée à la famille filtrante~$\Gamma_X$ des jeux de cartes tropicales
couvrants sur~$X$.
\end{proof}

\section{Formes positives}

Soit $X$ un espace G-tropical affiné.
Soit $\mathscr G$ une classe de coefficients.

\begin{lemm}
Soit $\omega$ une forme de type~$(p,p)$ et symétrique sur~$X$
qui est positive en tant que G-forme (\resp faiblement positive,
\resp fortement positive).
Pour tout point~$x\in X$, il existe 
une carte tropicale~$f$ définie sur un voisinage~$U$ de~$x$ 
à valeurs dans un espace affine~$E$, 
et une forme positive~$\alpha$ (\resp faiblement positive, \resp fortement positive) sur un sous-espace paralinéaire 
de~$E$ contenant~$f(U)$,
telle que $\omega|_U=f^*\alpha$.
\end{lemm}
\begin{proof}
Soit $x$ un point de~$X$. Soit $V$ un voisinage domanial compact de~$x$
sur lequel $\omega$ est tropicale;
soit $f\colon V\to E$ une carte tropicale qui tropicalise~$\omega$ et
soit~$\alpha$ une forme sur un sous-espace paralinéaire de~$E$
contenant~$f(V)$ tels que $\omega|_V=f^*\alpha$.
D'après le lemme~\ref{lemme-gpositif.compatibilite},
la restriction à~$f(V)$ de~$\alpha$
est positive (\resp faiblement positive, \resp fortement positive).
Il suffit donc de prendre pour carte tropicale la restriction à~$\mathring V$
de~$f$ et pour forme~$\alpha$ sa restriction à~$f(V)$.
\end{proof}

\begin{prop}\label{prop.trop-diff-positives}
Soit $\omega$ une forme de type~$(p,p)$ et symétrique sur~$X$,
à coefficients dans~$\mathscr G$ dont le support
est contenu dans un ouvert paracompact 
et localement différentiellement séparé de~$X$.
\begin{enumerate}
\item Il existe des formes positives
$\omega_1$ et~$\omega_2$ sur~$X$, à coefficients dans~$\mathscr G$,
telles que $\omega=\omega_1-\omega_2$.
\item Si, de plus, le support de~$\omega$ est compact,
les formes~$\omega_1$ et~$\omega_2$ peuvent être choisies
à support compact.
\end{enumerate}
\end{prop}
L'hypothèse sur le support de~$\omega$ est notamment réalisée 
lorsque $X$ est localement différentiellement séparé et paracompact.
Dans le cas des espaces analytiques, elle sera également réalisée
si le support de~$\omega$ est compact et si l'espace~$X$ est presque bon.
\begin{proof}
Soit $U$ un voisinage ouvert paracompact et
localement différentiellement séparé du
support de~$\omega$.  Dans le second cas, on s'arrange en outre
pour que $U$ soit relativement compact dans~$X$,
ce qui est possible car $X$ possède une base d'ouverts qui sont dénombrables
à l'infini.

Soit $(f_i,P_i)$ une famille de cartes tropicales,
de sorte que $f_i$ soit un moment sur un ouvert~$V_i$
de~$X$ et $P_i$ un polytope compact contenant $f_{i,\trop}(V_i)$,
et, pour tout~$i$, soit
$\alpha_i$ une forme sur~$P_i$ telle que 
$\omega\restr{V_i}=f_i^*\alpha_i$.
Il existe alors
une famille localement finie $(\lambda_i)$
de fonctions lisses sur~$X$,
subordonnée au recouvrement $(U\cap V_i)_{i\in I}$
telles que $0\leq\lambda_i\leq 1$
pour tout~$i$ et $\sum\lambda_i$ est  identiquement égale à~$1$
au voisinage du support de~$\omega$
(corollaire~\ref{prop.trop-partition.1.lisse-compact}).
Pour tout~$i$, choisissons des formes positives
$\alpha_{i,1}$ et $\alpha_{i,2}$ sur $P_i$ telles que
$\alpha_i=\alpha_{i,1}-\alpha_{i,2}$
(proposition~\ref{prop.diff-positive}).
Posons alors $\omega_1=\sum \lambda_i f_i^*\alpha_{i,1}$
et $\omega_2=\sum\lambda_i f_i^*\alpha_{i,2}$;
ce sont des formes lisses et positives
telles que $\omega=\omega_1-\omega_2$.
\end{proof}

\begin{rema}
Nous appliquerons cette proposition aux formes lisses.
Elle s'applique aussi aux G-formes lisses sur un espace tropical affiné décent,
car de telles G-formes sont des formes à coefficients G-lisses
(proposition~\ref{prop.local-G-local}).
\end{rema}

\begin{prop}\label{prop.riche-support}
Soit $X$ un espace G-tropical de niveau~$n$, riche. 
Soit $\omega$ une G-forme à coefficients tropicalement continus sur~$X$.

\begin{enumerate}
\item
Pour que $\omega$ soit nulle, resp. positive, resp. faiblement positive,
resp. fortement positive, il faut et il suffit
que pour toute partie squelettique compacte~$P$ de~$X$, 
la restriction~$\omega|_P$ de~$\omega$ à~$P$
soit nulle, resp. positive, resp. faiblement positive, 
resp. fortement positive.
\item
Les points squelettiques~$x$ du support de~$\omega$  
sont denses dans~$\supp(\omega)$.
\end{enumerate}
\end{prop}
\begin{proof}
\begin{enumerate}
\item
Si $\omega$ est nulle (resp\dots), il en est de même de ses
restrictions~$\omega|_P$ à toute partie squelettique comme dans l'énoncé.
On démontre la réciproque par contraposition.
Supposons que $\omega$  ne soit pas nulle
et démontrons qu'il existe une partie polyédrale~$P$ de~$X$,
compacte, purement de dimension~$n$, telle que $\omega|_P$
ne soit pas nulle. (Les cas positif, faiblement positif, 
fortement positif se démontrent de façon similaire.)

Par définition d'une G-forme, il existe un
G-recouvrement de~$X$ par des domaines compacts de~$X$
sur lesquels $\omega$ est G-tropicale.
Puisque $\omega$ n'est pas nulle, il existe un domaine compact~$V$ de~$X$
sur lequel $\omega$ est G-tropicale et non nulle.
Soit $f\colon V\to E$ une carte G-tropicale qui G-tropicalise~$\omega$,
soit $\alpha$ une forme sur~$f(V)$ telle que $\omega|_V=f^*\alpha$.
Comme $\omega|_V\neq0$, la forme~$\alpha$ n'est pas nulle.

Comme $X$ est riche, il existe une partie squelettique compacte~$\Sigma$
de~$V$ telle que $f(\Sigma)=f(V)$. En appliquant la
proposition~\ref{prop.nullite-formes} 
(lemme~\ref{lemme-gpositif.compatibilite} 
dans le cas des formes positives) à l'espace paralinéaire~$\Sigma$, considéré 
comme espace G-tropical, on en déduit que $\omega|_\Sigma\neq0$.

\item
Soit $U$ un ouvert de~$X$ rencontrant le support de~$\omega$;
on a $\omega|_U\neq0$. 
D'après~\emph a), il existe une partie squelettique compacte~$P$ de~$U$
telle que $\omega|_P\neq0$. Le support de~$\omega|_P$ est non vide,
il est contenu dans le support de~$\omega$
et il est constitué de points squelettiques.
\qedhere
\end{enumerate}
\end{proof}

\begin{coro}\label{coro.tres-riche-support}
Soit $X$ un espace tropical de niveau~$n$, très riche.
Soit $\omega$ une G-forme à coefficients tropicalement continus sur~$X$.

\begin{enumerate}
\item
Pour que $\omega$ soit nulle, resp. positive, resp. faiblement positive,
resp. fortement positive, il faut et il suffit
que pour toute partie squelettique compacte~$P$ de~$X$, 
purement de dimension~$n$,
la restriction~$\omega|_P$ de~$\omega$ à~$P$
soit nulle, resp. positive, resp. faiblement positive, 
resp. fortement positive.
\item
Les points~$x$ du support de~$\omega$  
tels que $d_\trop(x)=n$
sont denses dans~$\supp(\omega)$.
\end{enumerate}
\end{coro}
\begin{proof}
\begin{enumerate}
\item
Une direction est évidente; supposons que $\omega$ ne soit pas
nulle (les autres cas sont analogues).
L'espace~$X$ est riche ;
d'après la proposition précédente, il existe une partie squelettique
compacte~$\Sigma$ de~$X$ telle que $\omega|_\Sigma\neq0$. 
Comme $X$ est très riche, il existe une partie squelettique compacte~$\Sigma'$
de~$X$, purement de dimension~$n$,
telle que $\Sigma\subset\Sigma'$; on a donc $\omega|_{\Sigma'}\neq0$.

\item
Soit $U$ un ouvert de~$X$ qui rencontre le support de~$\omega$;
soit $x$ un point squelettique de $U$ qui appartient à~$\supp(\omega)$.
Comme $X$ est très riche, il existe une partie squelettique~$\Sigma$ de~$U$,
purement de dimension~$n$, telle que $x\in U$.
Alors, $d_\trop(x)=n$. \qedhere
\end{enumerate}
\end{proof}

\section{Espaces tropicaux}\label{s.espace-tropical}

\begin{defi}\label{defi.esp-harmonieux}
On dit qu'un espace G-tropical calibré et affiné~$X$ est harmonieux 
si pour tout ouvert~$U$ de~$X$ 
et toute carte tropicale $f\colon U\to E$ sur~$U$,
le calibrage canonique de $\Sigma_f^{(n)}$ est harmonieux relativement à~$f$.
\end{defi}

\begin{defi}\label{defi.esp-tropical}
Un espace tropical de niveau~$n$
est la donnée d'un espace G-tropical calibré et affiné~$X$ de niveau~$n$,
riche, muni d'une partie fermée~$\partial(X)$.
On fait en outre les hypothèses suivantes sur l'ouvert~$\Int(X)$,
complémentaire de~$\partial(X)$:
\begin{enumerate}\def\theenumi{\roman{enumi}}\def\labelenumi{(\theenumi)}
\item L'ouvert $\Int(X)$ est localement différentiellement séparé,
décent, très riche et harmonieux.
\item Pour toute  partie squelettique~$\Sigma$ de~$X$
et tout domaine~$V$ de~$X$, 
la trace de~$\partial(X)$ sur~$\Sigma$ 
et la trace du bord topologique~$\partial(V/X)$ sur~$\Sigma$
sont $n$-négligeables.
\end{enumerate}
\end{defi}

On dit que $\partial(X)$ est le \emph{bord} de l'espace tropical~$X$
et que $\Int(X)$ est son \emph{intérieur}.

\subsection{}
Soit $X$ un sous-espace paralinéaire de dimension~$n$ 
d'un espace affine~$E$.
Soit $\mathscr C$ une décomposition cellulaire de~$X$
en cellules fortes de~$E$ et munissons~$X$
d'un calibrage qui est constant sur chaque cellule de~$\mathscr C$.

Alors, $X$ est un espace tropical de niveau~$n$
et de bord  la réunion des $(n-1)$-cellules de~$\mathscr C$.

\subsection{}
Soit $X$ un espace tropical de niveau~$n$ de bord~$\partial(X)$.
Soit $V$ un domaine de~$X$.
Posons $\partial(V)= \partial(V/X)\cup (V \cap \partial(X))$.
Alors $V$ est un espace tropical de niveau~$n$ de bord~$\partial(V)$.

Par construction, $V\setminus \partial(V)$ 
est un ouvert de $X\setminus \partial(X)$ ; c'est donc
un espace G-tropical vertébré très riche, un espace G-tropical affiné décent
et un espace G-tropical calibré et affiné harmonieux.

Soit $\Sigma$ une partie squelettique de~$V$.
La trace de~$\partial(V)$ sur~$\Sigma$ 
est la réunion de deux parties $n$-négligeables, donc est
$n$-négligeable. 
Soit $W$ un domaine de~$V$. Le bord topologique
$\partial(W/X)$ est égal à $\partial(W/V) \cup (W \cap \partial (V/X))$;
sa trace sur~$\Sigma$ est $n$-négligeable par hypothèse,
si bien que la trace de $\partial(W/V)$ est $n$-négligeable.


Cela prouve que $V$ est un espace tropical de niveau~$n$
de bord~$\partial(V)$.

\subsection{}
Le cœur de notre travail consistera à munir tout espace analytique~$X$
de dimension~$n$ d'une structure d'espace tropical de niveau~$n$
dont le bord est donné par la théorie de Berkovich.

\begin{prop}
\label{theo.support-mesure-bord}
Soit $X$ un espace tropical de niveau~$n$.
Pour toute $(n-1,n)$-forme~$\omega$ sur~$X$, 
le support de la mesure de bord~$\tilde\omega$ est contenu dans~$\partial(X)$.
\end{prop}
\begin{proof}
Comme $\Int(X)$ est ouvert dans~$X$,
il s'agit de démontrer que $\tilde\omega|_{\Int(X)}=0$.
Comme $\tilde\omega|_{\Int(X)}=\widetilde{\omega|_{\Int(X)}}$,
on peut remplacer~$X$ par~$\Int(X)$ et supposer que $X$ est sans bord;
il s'agit de démontrer que $\tilde\omega=0$.

L'assertion est locale sur~$X$. 
Comme $\omega$ est une forme, on peut supposer qu'il existe
une carte tropicale $f\colon X\to E$ et une $(n-1,n)$-forme~$\alpha$
sur un sous-espace paralinéaire de~$E$ contenant~$f(X)$
telles que $\omega=f^*\alpha$.
On peut aussi supposer que $X$ est paracompact.
Alors $\Sigma_f^{(n)}$ est un support fort de~$\omega$ qui est paracompact.

Considérons une décomposition cellulaire $\mathscr C$ de~$\Sigma_f^{(n)}$
formée de $f$-cellules qui est adaptée au  calibrage de~$\Sigma_f^{(n)}$.
La forme-volume $(n-1)$-dimensionnelle~$\tilde\omega$ 
est supportée par la réunion
des $(n-1)$-cellules de~$\mathscr C$. 
Soit donc $F$ une $(n-1)$-cellule de~$\mathscr C$.
Soit $\mathscr C_{>F}$ l'ensemble des $n$-cellules de~$\mathscr C$
contenant~$F$; il est fini et non vide, car $\Sigma_f^{(n)}$ est purement
de dimension~$n$.
Fixons une orientation~$o$ de~$f(F)$ et, pour tout~$C\in\mathscr C_{>F}$,
notons~$o_C$ l'orientation de $f(F)$ déduite de celle de~$f(C)$
par la règle de la normale sortante à~$f(F)$.
Il existe un unique $n$-vecteur, $\tilde\mu_C$, sur $\langle f(C)\rangle$
tel que $f_*(\mu_C)$ soit la classe du couple $(\tilde\mu_C, o_C)$.

La formule volume $(n-1)$-dimensionnelle $\tilde\omega$
se déduit d'une 
forme-volume $(n-1)$-dimensionnelle 
$\langle\omega,\mu_{\Sigma_f^{(n)}}\rangle$ 
sur $\Sigma_f^{(n)}$;
rappelons-en la construction (\S\ref{sss.forme-bord-construction}).
Par contraction
par le $n$-vecteur $\mu_C$ 
et multiplication par $(-1)^{n(n-1)/2}$,
on déduit de la forme~$\alpha|_{f(C)}$
une $(n-1,0)$-forme sur $f(C)$;
on la restreint à~$f(F)$, on la transforme
en forme-volume par le choix de l'orientation~$o$
et on la rapatrie sur~$F$ au moyen de l'isomorphisme $f|_F$.
On obtient ainsi une forme-volume $\tilde\omega_C$ sur~$F$.
Par définition, la restriction à~$F$ 
de la $(n-1)$-forme-volume~$\tilde\omega$
est la somme
$ \tilde\omega_F = \sum_{C\in\mathscr C_F} \tilde\omega_C$.
Ainsi, $\tilde\omega_F$ est déduite de l'orientation~$(-1)^{n(n-1)/2}o_F$
et de la  $(n-1)$-forme $\sum_C \langle\alpha,\tilde\mu_C\rangle
=\langle\alpha,\tilde\mu_F\rangle$ sur~$f_\trop(F)$,
où l'on a posé $\tilde\mu_F=\sum_{C\in\mathscr F}\tilde\mu_C$.
Cette expression n'est autre que la discordance de~$C$
relativement à la carte tropicale~$f$, donc est nulle par définition
du bord d'un espace tropical. Ainsi, $\tilde\omega_F=0$.
\end{proof}

\begin{coro}
Soit $X$ un espace tropical de niveau~$n$,
soit $\omega$ une G-forme de type $(n-1,n)$ sur~$X$
et soit $(X_i)$ un G-recouvrement de~$X$ par des domaines 
tels que pour tout~$i$, $\omega|_{X_i}$ soit une forme.
Alors, le support de~$\widetilde\omega$ est contenu dans
$\bigcup_i \partial(X_i)$.
\end{coro}

\begin{theo}[Condition d'harmonie] \label{theo.condition-harmonie}
Soit $X$ un espace tropical de niveau~$n$,
soit $f\colon X\to E$ une carte tropicale agréable.
Alors $\Pi_f$ est harmonieux en dehors de $f(\partial(X)\cap \Sigma_f^{(n)})$.
\end{theo}
\begin{proof}
Soit $\mathscr D$ une décomposition cellulaire de $\Pi_f$
adaptée au calibrage canonique~$\mu_{\Pi_f}$ et à~$\partial(X)$.
Soit $F$ une~$(n-1)$-cellule de~$\mathscr D$
qui n'est pas contenue dans $f(\partial (X)\cap\Sigma_f^{(n)})$.
La discordance de la face~$F$ 
est un $(0,n)$-vecteur~$\delta_F$ de~$E$; 
supposons par l'absurde qu'il n'est pas nul
et choisissons un multi-indice~$J$ de longueur~$n$ 
tel que $\langle \dc x_J,\delta_F\rangle\neq0$; 
choisissons aussi un multi-indice~$I$ de longueur~$n-1$
tel que la restriction à~$F$ de la $(n-1)$-forme $\di x_I$ ne soit pas nulle.

Soit $x$ un point de~$\mathring F$ qui n'appartient pas 
à $f(\partial(X))$.
Soit $u$ une fonction continue, non identiquement nulle
sur~$\Pi_f$ dont le support est compact, disjoint de~$f(\partial(X))$
et des faces de dimension~$n-1$ autres que~$f(F)$.
Posons $\alpha=u\di x_I\wedge\dc x_J$ et $\omega=f^*\alpha$.
Par construction, $\langle\alpha, \mu_{\Pi_f}\rangle$
est une forme-volume non nulle.

D'après la proposition~\ref{theo.support-mesure-bord},
le support de la forme-volume~$\tilde\omega$ 
est contenu dans~$\partial(X)$.
Il est aussi disjoint de~$\partial(X)$, par construction. 
On a donc $\tilde\omega=0$.
D'après la relation~\ref{n-1-forme-volume-but},
on a donc 
$\langle\alpha,\mu_{\Pi_f}\rangle =f_*(\tilde \omega)=0$,
ce qui contredit la construction de~$\alpha$.
\end{proof}

\section{Exemples}

\subsection{Segments}
Soit $I$ un intervalle de~$\R$.
Il possède une structure évidente d'espace paralinéaire,
qui en fait un espace G-tropical. Ses domaines
sont les réunions localement finies d'intervalles. 
Munissons-le du $1$-calibrage constant, donné par le vecteur-volume~$\abs{e_1}$.
Munissons-le également de la structure d'espace affiné pour laquel
une fonction sur un domaine connexe de~$I$ est affine si et seulement
si c'est la restriction d'une fonction affine sur~$\R$.
Alors $I$ est un espace tropical de niveau~$1$
dont le bord est la frontière de~$I$ dans~$\R$.

Si $1\notin I$, on peut réaliser cet espace comme la couronne $\abs T\in I$
sur un corps trivialement valué.

\subsection{Un autre type de segment}
Soit $I$ un intervalle de~$\overline\R$ 
et posons $A=\partial(I)\cap\overline\R$.
On munit~$I$ d'une structure d'espace G-tropical dont les domaines~$V$
sont les réunions localement finies d'intervalles qui sont
voisinages de $V\cap A$. Une fonction sur un tel domaine est dite
paralinéaire si elle l'est au sens usuel et si elle est constante
au voisinage de tout point de~$V\cap A$.
Une partie de~$I$ est squelettique si c'est une partie paralinéaire
de $I\cap\R$.
On munit~$I$ du $1$-calibrage constant donné par le vecteur-volume~$\abs {e_1}$.
On dit qu'une fonction paralinéaire sur un domaine connexe~$V$
de~$I$ est affine si c'est la restriction à~$V$ d'une fonction affine
sur~$I$. En particulier, si $V$ est un domaine connexe de~$I$
qui rencontre~$A$, toute fonction affine sur~$V$ est constante.

Cet espace G-tropical est un espace tropical de niveau~$1$,
de bord~$A$.

Lorsque $I=[-\infty;0\mathclose[$, on peut réaliser cet exemple 
comme le disque unité ouvert sur un corps trivialement valué.

\subsection{Une étoile}\label{defi-etoile}
Soit $B$ un ensemble. On considère l'ensemble~$X_B$
quotient de l'ensemble produit $B\times[-\infty;0]$
par la relation d'équivalence qui identifie 
tous les points de la forme~$(b,0)$ à un même point~$\eta$.
On le munit de la topologie la moins fine
pour laquelle une partie~$V$ de~$X_B$ est ouverte
si et seulement si d'une part $p^{-1}(V)\cap (\{b\}\times [-\infty;0])$
est ouverte dans~$\{b\}\times[-\infty;0]$ pour tout $b\in B$,
et d'autre part, si $\eta\in V$, alors l'ensemble des $b\in B$
tels que $p^{-1}(V)$ ne contient pas $\{b\}\times[-\infty;0]$
est fini. L'application~$p$ est continue. L'espace~$X_B$ est compact.
Lorsque $B$ est infini, la topologie de~$X_B$ est strictement moins fine
que la topologie quotient.

Pour $b\in B$, l'application $t\mapsto p(b,t)$ induit un homéomorphisme
de~$[-\infty;0]$ sur une partie~$X_b$ de~$X_B$ que l'on appelle
la branche d'indice~$b$.
Le point~$\eta$ est le centre de l'étoile~$X_B$.

Les domaines de~$X_B$ sont définis de manière analogue:
une partie~$V$ de~$X_B$ est un domaine si $V_b=\{t\sozat (b,t)\in V\}$
est un domaine de $\{b\}\times[-\infty;0]$ et si, 
dès que $\eta\in V$, l'ensemble des $b\in B$
tels que $V_b\neq [-\infty;0]$ est fini.

Les parties squelettiques de~$X_B$ sont également définies
de la sorte: une partie~$S$ de~$X_B$ est squelettique si
$S\cap X_b$ est une partie paralinéaire de~$X_b\simeq[-\infty;0]$,
c'est-à-dire une partie paralinéaire de~$\mathopen]-\infty;0]$
et si, dès que $\eta\in S$, l'ensemble des $b\in B$
tels que $S\cap X_b\neq\{\eta\}$ est fini. 

On calibre chaque branche~$X_b$ à l'aide du vecteur-volume~$\abs{e_1}$.

Une fonction~$f$ sur un domaine~$V$ est dite paralinéaire
si, pour tout $b\in B$, la fonction $f_b\colon t\mapsto f(p(b,t))$
sur le domaine~$V_b$
est paralinéaire, constante au voisinage de~$-\infty$ si $-\infty\in V_b$,
et si, dès que $\eta\in V$,
il existe une partie finie~$B_1$ de~$B$ telle que 
pour tout $b\in B\setminus B_1$, la fonction~$f_b$
est définie et constante sur~$[-\infty;0]$.

Une fonction paralinéaire~$f$ sur un domaine~$V$ est dite affine
si, pour tout~$b\in B$, la fonction~$f_b$ sur~$V_b$
est affine, et si, 
dès que $\eta\in V$,
la somme (presque nulle) des~$f_b'(0)$, pour $b\in B$
tel que $V_b$ soit un voisinage de~$0$ dans~$[-\infty;0]$,
soit égale à zéro.

Alors $X_B$ est un espace tropical de niveau~$1$, sans bord.

Soit $f$ une fonction sur un ouvert~$U$ de~$X$. Pour que $f$ soit lisse,
il faut et il suffit que chacune des fonctions~$f_b$ sur
l'ouvert~$U_b=\{t\sozat p(b,t)\in U\}$ de~$[-\infty;0]$ soit lisse,
constante au voisinage de~$-\infty$ si $-\infty\in U_b$,
et que si $\eta\in U$,
l'ensemble des~$b\in B$ tels que $f_b$ ne soit pas définie et constante 
sur~$[-\infty;0]$ est fini, et que la somme (presque nulle), pour $b\in B$
tel que~$U_b$ est voisinage de~$0$, des $f_s'(0)$ soit nulle.

On peut réaliser~$X_B$ comme l'espace analytique $\P^1_k$
sur un corps algébriquement clos trivialement valué;
l'ensemble~$B$ s'identifie à~$\P^1(k)$
et le point~$\eta$ à la valeur absolue triviale sur
le corps des fonctions~$k(\P^1_k)$.

\subsection{}
On peut construire une variante de l'étoile précédente
à partir d'une famille d'intervalles $[-\infty;0]$
indexée par un ensemble~$B$ et d'une famille d'intervalles $[-1;0]$
indexée par un ensemble~$A$, en recollant les origines en un même point~$\eta$,
en lequel on peut, ou non, imposer la condition d'harmonie.

Lorsque $A$ est fini, cet espace s'interprète comme un domaine
de l'étoile associée à~$A\cup B$.

Lorsque $A$ est infini, le point~$\eta$ appartient au bord de l'espace tropical
obtenu, même si la condition d'harmonie y est imposée,
car $\eta$ est alors adhérent à l'ensemble des points 
de la forme~$p(a,0)$, pour $a\in A$.

\subsection{}
Soit $(X_i, f_{ij})$ un système projectif d'espaces tropicaux de même niveau~$n$,
indexé par un ensemble ordonné filtrant à droite~$I$
et soit $X=\varprojlim X_i$ sa limite projective.
Pour tout~$i$, on note $f_i\colon X\to X_i$ l'application canonique.

On fait l'hypothèse que pour $i\leq j$, le morphisme
$f_{ij}\colon X_j\to X_i$ est surjectif et induit un isomorphisme
entre parties squelettiques:
pour tout domaine~$V$ de~$X_i$ et toute partie squelettique~$S$
de~$X_i$, 
$f_{ij}^{-1}(S)$ est une partie squelettique du domaine~$f_{ij}^{-1}(V)$
telle que $f_{ij}$ induit un homéomorphisme de $f_{ij}^{-1}(S)$ sur~$S$,
et toute partie squelettique de $f_{ij}^{-1}(V)$ est de cette forme.

On peut alors munir alors~$X$ d'une structure d'espace tropical de niveau~$n$
de la façon suivante.

Par définition de la topologie de~$X$ et surjectivité des applications~$f_{ij}$,
une partie~$U$ de~$X$ est ouverte s'il existe un indice~$i$
et une partie ouverte~$U_i$ de~$X_i$ telle que $U=f_i^{-1}(X_i)$.

On dit qu'une partie~$V$ de~$X$ est un domaine s'il existe un indice~$i$
et un domaine~$V_i$ de~$X_i$ tel que $V=f_i^{-1}(V_i)$.

On dit qu'une fonction~$u$ sur un domaine~$V$ de~$X$
est paralinéaire (resp. affine) s'il existe un indice~$i$,
un domaine~$V_i$ de~$X_i$ tel que $V=f_i^{-1}(V_i)$
et une fonction paralinéaire (resp. affine)~$u_i$ sur~$V_i$
telle que $u=u_i\circ f_i$.

Avec ces définitions, une fonction~$u$ sur un domaine~$U$ de~$X$ est lisse 
s'il existe un indice~$i$, un domaine~$U_i$ de~$X_i$
tel que $U=f_i^{-1}(U_i)$
et une fonction lisse~$u_i$ sur~$U_i$ telle que $u=u_i\circ f_i$.

On dit qu'une partie~$S$ d'un domaine~$V$ de~$X$ 
est squelettique s'il existe un indice~$i$, un domaine~$V_i$
de~$X_i$ tel que $V=f_i^{-1}(V_i)$ et une partie squelettique~$S_i$ de~$X_i$
telle que $f_i\colon X\to X_i$ induise un homéomorphisme
de~$S$ sur~$S_i$. Soit alors $j$ un indice tel que~$j\geq i$;
posons~$V_j=f_{ij}^{-1}(V_i)$ : c'est un domaine de~$X_j$
tel que $V=f_j^{-1}(V_j)$.
Par hypothèse, $S_j=f_{ij}^{-1}(S_i)$ est une partie squelettique de~$V_j$
et $f_{ij}$ induit un isomorphisme paralinéaire de~$S_j$ sur~$S_i$, compatible
aux calibrages.
Ainsi, $f_j$ induit un homéomorphisme de~$S$ sur~$S_j$.
On utilise ces homéomorphismes paralinéaire 
pour calibrer la partie squelettique~$S$.

On peut appliquer cette construction en partant d'un espace analytique~$X$,
sur un corps complet archimédien~$k$ et
d'un ouvert de Zariski~$U$ de~$X$,
en prenant
pour famille~$(X_i)$ les morphismes propres $h_i\colon X_i\to X$,
qui sont un isomorphisme au-dessus d'un ouvert dense de~$X$ contenant~$U$.
La limite projective $\widetilde X$ est un espace tropical
sur lequel on peut interpréter les constructions de~\cite{YuanZhang-2025}.

\chapter{Courants sur un espace tropical}

\section{Courants}
\label{ss.courants}

Soit $X$ un espace tropical de niveau~$n$.

\subsection{}
On appelle \emph{courant} de bidimension~$(p,q)$ sur~$X$ une forme
linéaire continue sur l'espace~$\mathscr A_\cpct^{p,q}(X)$ des formes
lisses de type~$(p,q)$ à support compact, lorsque
cet espace est muni de la topologie induite
par la topologie~$\mathscr C^\infty$ 
sur l'espace des formes lisses à support compact
définie au paragraphe~\ref{subsec.topologie-formes-lisses}

On dit qu'un courant est d'ordre~$\leq r$ s'il est continu
pour la topologie~$\mathscr C^r$ sur l'espace des formes (il est alors
d'ordre~$\leq s$ pour tout $s\geq r$).

On notera~$\langle S,\alpha \rangle$ l'image de la forme~$\alpha$
par le courant~$S$.

On note~$\mathscr D_{p,q}(X)$ l'ensemble des courants
de bidimension~$(p,q)$.   
On le munit de la topologie de la convergence simple.

On posera aussi~$\mathscr D^{p,q}=\mathscr D_{n-p,n-q}$
et on dira que c'est l'espace des courants de bidegré~$(p,q)$.

\begin{rema}\label{rema.continuite-courant}
Explicitons la condition de continuité
dans la définition d'un courant.
L'espace $\mathscr A_{\cpct}^{p,q}(X)$ est la limite
inductive des espaces $\mathscr A^{p,q}_{\mathscr D,K}(X)$,
où $\mathscr D$ parcourt l'ensemble des jeux de cartes tropicales
couvrants sur~$X$ et $K$ l'ensemble des parties compactes de~$X$.
Par définition, une forme linéaire 
sur l'espace $\mathscr A_{\cpct}^{p,q}(X)$ est continue
si et seulement si sa restriction à chaque sous-espace $\mathscr A^{p,q}_{\mathscr D,K}(X)$ est continue.

Fixons une telle donnée $(\mathscr D,K)$.
L'espace $\mathscr A^{p,q}_{\mathscr D,K}(X)$ est muni
de la topologie induite de $\mathscr A^{p,q}_{\mathscr D}(X)$,
laquelle est définie par les classes de semi-normes $\norm\cdot_{W,r}$
associées aux domaines compacts~$W$ d'une carte~$g$ de~$\mathscr D$
et à un entier~$r$.
(Pour fixer une semi-norme dans chaqque classe,
il faut choisir une décomposition cellulaire
de $g(W)$ et des systèmes de coordonnées affines sur chaque cellule.)

La continuité de~$S$ sur cet espace $\mathscr A^{p,q}_{\mathscr D,K}(X)$
signifie donc qu'il existe un nombre réel~$c$, un entier~$r$ et 
une famille finie $(W_i)$ de domaines
de compacts de~$X$, où pour tout~$i$, $W_i$ est contenu
dans le domaine d'une carte de~$\mathscr D$, tels que,
pour toute forme lisse $\omega\in \mathscr A^{p,q}_{\mathscr D,K}(X)$,
on ait
\[ \abs{\langle S, \omega\rangle} \leq c \sup_i \norm{\omega}_{W_i,r}. \]

Comme l'espace tropical~$X$ est riche, 
la construction faite dans la remarque~\ref{rema.topologie-formes-squelettes}
fournit une description squelettique de la continuité de~$S$:
il existe un nombre réel~$c$, un entier~$r$ et une partie squelettique
compacte~$\Sigma$ de~$X$ telle que,
pour toute forme lisse 
$\omega\in \mathscr A^{p,q}_{\mathscr D,K}(X)$,
on ait
\[ \abs{\langle S, \omega\rangle} \leq c \norm{\omega|_{\Sigma}}_{\Sigma,r}, \]
où $\norm\cdot_{\Sigma,r}$ est une norme~$\mathscr C^r$ standard sur~$\Sigma$.

On peut résumer cette discussion en disant qu'un courant
possède un ordre relatif à tout jeu de cartes couvrant et toute partie 
compacte de~$X$.
L'ordre construit dépend de la partie compacte considérée,
comme dans le cas des distributions réelles.

On prendra garde que l'ordre construit dépend également
du jeu de cartes considéré.
On peut en effet construire un courant sur un espace
tropical compact pour lequel ces ordres ne sont pas majorés.
\end{rema}

\begin{lemm}\label{lemm.continuite-courant}
Supposons que $X$ soit localement différentiellement séparé.
Soit $\mathscr B$ un jeu de cartes sur~$X$ 
tel que tout jeu  de cartes couvrant sur~$X$ 
admet un raffinement couvrant formé de cartes appartenant à~$\mathscr B$.
Soit $S$ une application linéaire de $\mathscr A_\cpct^{p,q}(X)$
à valeurs dans espace vectoriel topologique localement convexe~$E$.
Soit $r$ un élément de~$\N\cup\{+\infty\}$.
Pour que $S$ soit continue pour la topologie~$\mathscr C^ r$ il faut 
et il suffit que pour toute carte tropicale~$f$ 
définie sur un ouvert~$U$ de~$X$ appartenant à la base~$\mathscr U$
et toute fonction~$\rho$ sur~$U$,
lisse, positive et à support compact, 
la restriction à~$\mathscr A^{p,q}_f(U)$
de l'application linéaire $\omega\mapsto S(\rho\omega)$
soit continue pour la topologie~$\mathscr C^r$.
\end{lemm}
\begin{proof}
Soit $\mathscr D=(f_i\colon U_i\to E_i)$ 
un jeu de cartes couvrant sur~$X$ 
et soit $K$ une partie compacte de~$X$; 
il s'agit de prouver que la restriction de~$S$
à l'espace $\mathscr A^{p,p}_{\mathscr D,K}(X)$ est continue
pour sa topologie~$\mathscr C^\infty$ (\resp $\mathscr C^r$).
Par hypothèse, il existe un jeu de cartes~$\mathscr D'$, couvrant,
formé de cartes de~$\mathscr B$ qui  raffine le jeu~$\mathscr D$.
Cela permet de supposer que les cartes de~$\mathscr D$
appartiennent à la base d'ouverts~$\mathscr U$.
On choisit une partie finie~$I$ de l'ensemble des indices~$i$ 
telle $K\subset \bigcup_{i\in I} \Int(U_i)$.
Soit $(\lambda_i)_{i\in I}$ une famille de fonctions lisses et positives
sur~$X$ telles que pour tout~$i$, le support de~$\lambda_i$
soit compact et contenu dans~$\Int(U_i)$,
et telles que $\sum_{i\in I}\lambda_i$ soit égale à~$1$ au voisinage de~$K$.
Pour toute forme $\omega\in\mathscr A^{p,q}_{\mathscr D,K}$,
on a donc $\omega = (\sum_{i\in I}\lambda_i)\omega$, de sorte
que $ S(\omega) = \sum_{i\in I} S(\lambda_i\omega)$.
Il suffit ainsi de prouver que chacune des applications linéaires
$\omega\mapsto S(\lambda_i\omega)$ est continue sur $\mathscr A_{\mathscr D,K}$.

Soit $i\in I$.
L'application de restriction à~$U_i$ induit 
une application de~$\mathscr A^{p,q}_{\mathscr D,K}(X)$
dans~$\mathscr A^{p,q}_{f_i}(U_i)$; cette application est continue
pour les topologies~$\mathscr C^m$.
Par hypothèse, l'application linéaire $\alpha\mapsto
S(\lambda_i\alpha)$ est continue sur~$\mathscr A^{p,q}_{f_i}(U_i)$,
pour la topologie~$\mathscr C^\infty$ (\resp.~$\mathscr C^r$).
La composée de ces applications linéaires est donc continue.
\end{proof}

\subsection{}
Soit~$T$ un courant de bidimension~$(p,q)$;
il découle de la proposition~\ref{prop.continuite-di-dc-J'}
que l'application $T\circ\mathrm J^{-1}$ est un 
courant de bidimension~$(q,p)$ que l'on note $\mathrm JT$.
Cet opérateur~$\mathrm J$ sur les courants est continu.

On dit qu'un courant~$T$ de bidimension~$(p,p)$ 
est symétrique\index{courant!--- symétrique}
si l'on a~$\mathrm JT=T$.  

\subsection{}
Soit~$T$ un courant et soit~$\alpha$ une forme lisse sur~$X$.
Il découle de la proposition~\ref{prop.continuite-wedge}
que l'application $\beta\mapsto \langle T, \alpha\wedge\beta\rangle$
est un courant  ; on le note $T\wedge\alpha $.
Il est d'ordre~$r$ si $T$ est d'ordre~$\leq r$.

L'application $T \mapsto T\wedge\alpha$ est continue.

On observe la relation 
$\mathrm J(T\wedge\alpha)=\mathrm JT\wedge\mathrm J\alpha$.

Si~$\alpha\in\mathscr A^{p,q}_\cpct(X)$
et~$T$ est de bidimension~$(p',q')$, le courant
$T\wedge\alpha $ est de bidimension~$(p'-p,q'-q)$;
si~$X$ est équidimensionnel et si
$T$ est de bidegré~$(p',q')$, alors~$T\wedge \alpha$ est de bidegré 
$(p'+p,q'+q)$.

Avec ces notations, 
on définit aussi $\alpha\wedge T = (-1)^{(p+q)(p'+q')} T\wedge\alpha$.

Munis de ces lois d'action,
les espaces de courants sont des modules 
sur l'anneau des fonctions lisses sur~$X$.
Plus précisément, l'espace vectoriel bigradué~$\mathscr D(X)$,
somme directe des espaces~$\mathscr D_{p,q}(X)$,
est un module gradué (à gauche et à droite) 
sur l'algèbre graduée~$\mathscr A(X)$ 
des formes lisses sur~$X$. 

\subsection{Propriété de faisceau}
Soit $T$ un courant sur~$X$.
Soit~$U$ un ouvert de~$X$ et soit $j$ l'immersion de~$U$ dans~$X$.
Comme l'opération~$j_!$  de prolongement par~$0$
des formes à support compact sur~$U$ est
une application linéaire continue (prop.~\ref{prop.continuite-push-U'}),
l'application $\alpha\mapsto \langle T,j_!\alpha\rangle$
est un courant~$T|_U$ sur~$X$ qu'on appelle restriction de~$T$ à~$U$.
Si $T$ est d'ordre~$\leq r$, il en est de même de~$T|_U$.

L'application~$U\mapsto\mathscr D_{p,q}(U)$ est ainsi un préfaisceau,
et pour tout entier~$r$,
les courants d'ordre~$\leq r$ en forment un sous-préfaisceau.

Pour tout courant~$T$ et tout ouvert~$U$,
on a $\mathrm JT|_U=(\mathrm JT)|_U$.

\begin{prop}\label{prop.courants-faisceau}
Supposons que l'espace tropical~$X$ soit localement différentiellement séparé.
Le préfaisceau~$\mathscr D_{p,q}$ sur~$X$ est 
un faisceau de~$\mathscr A^{0,0}$-modules.
Pour tout entier~$r$, le sous-préfaisceau des courants d'ordre~$\leq r$
en est un sous-faisceau.
\end{prop}
En particulier, l'ordre d'un courant est  une notion locale.
\begin{proof}
Soit~$(U_i)$ un recouvrement ouvert d'un ouvert~$U$ de~$X$,
pour tout~$i$, soit~$S_i$ un courant sur~$U_i$ tels
que la restriction à~$U_i\cap U_j$ des courants~$S_i$ et~$S_j$
coïncident, pour tous~$i,j$.
Démontrons qu'il existe un unique courant~$S$ sur~$U$
dont la restriction à~$U_i$ est égale à~$S_i$.

Soit~$\omega$ une~$(p,q)$-forme à support compact sur~$U$.
Soit~$(\phi_i)$ une famille 
finie de fonctions lisses à support compact sur~$U$
telle que le support de~$\phi_i$ est contenu dans~$U_i$
et telle que~$\sum\phi_i$ soit égale à~$1$ au voisinage du support
de~$\omega$ (corollaire~\ref{prop.trop-partition.1.lisse-compact}).
On a~$\omega=\sum \phi_i \omega$, et seuls un nombre fini
de termes de cette somme sont non nuls.
Si~$S$ est un courant sur~$U$ dont la restriction à~$U_i$
est égale à~$S_i$, on a donc nécessairement
$\langle S,\omega\rangle =\sum \langle S_i,\phi_i\omega\rangle$,
d'où l'unicité d'un courant~$S$ vérifiant les conditions requises.

Soit~$(\phi'_i)$ une seconde famille finie
de fonctions lisses à support compact sur~$U$
telle que le support de~$\phi'_i$
soit contenu dans~$U_i$
et que~$\sum\phi'_j$ soit égale à~$1$ au voisinage du support
de~$\omega$.
On a aussi~$\omega=\sum_{i,j} \phi'_j\phi_i \omega$, si bien que
\[
\sum_j \langle  \phi'_j\omega , S_j\rangle
 = \sum_{i,j} \langle \phi'_j\phi_i \omega ,S_j\rangle 
= \sum_{i,j} \langle \phi'_j\phi_i\omega,S_j\rangle
 = \sum_{i} \langle \phi_i\omega,S_i\rangle ,\]
où la deuxième égalité découle de la condition de coïncidence
des courants~$S_i$.

Il existe donc une unique application~$S$ de l'espace~$\mathscr 
A^{p,q}_\cpct(U)$
dans~$\R$ telle que
$\langle S,\omega\rangle=\sum \langle S_i,\phi_i\omega\rangle$
pour tout famille $(\phi_i)$ de fonctions lisses à support
compact telle que $\sum\phi_i$ soit égale à~$1$ au voisinage
du support de~$\omega$.
Cette application est linéaire. 

\medskip

Démontrons que c'est un courant  sur~$U$. 
Soit $K$ une partie compacte de~$U$.
Considérons alors une famille  finie
$(\phi_i)$ de fonctions lisses à support compact dont la 
somme est égale à~$1$ au voisinage de~$K$ 
et telle que le support de~$\phi_i$
soit contenu dans~$U_i$.
Pour tout~$i$, 
l'application $\omega\mapsto \phi_i\omega$ 
de $\mathscr A^{p,q}_K(U)$ dans~$\mathscr A^{p,q}_\cpct(U_i)$
est continue,
donc
la forme linéaire $\omega\mapsto \langle S_i,\phi_i\omega\rangle$ 
sur $\mathscr A^{p,q}_K(U)$ 
est continue. Pour tout $\omega\in\mathscr A^{p,q}_K(U)$,
on a $\langle S,\omega\rangle = \sum_i \langle S_i,\phi_i\omega\rangle$.
Par suite, la restriction de~$S$
à $\mathscr A^{p,q}_K(U)$ est continue.
Cela prouve que $S$ est continue; c'est donc un courant sur~$U$.

\medskip

Soit~$\omega$ une~$(p,q)$-forme dont le support est compact
et contenu dans~$U_j$; soit~$(\phi_i)$ une famille comme ci-dessus.
Pour tout~$i$, le support de~$\phi_i\omega$ est contenu dans~$U_i\cap U_j$,
donc~$\langle S_i,\phi_i\omega\rangle =\langle S_j,\phi_i\omega,S_i\rangle$.
Alors, 
\[ 
\langle S,\omega \rangle = \sum \langle S_i,\phi_i\omega\rangle
= \sum\langle S_j,\phi_i\omega\rangle
= \langle S_j,\omega\rangle \]
puisque $\sum\phi_i$ est égal à~$1$ au voisinage du support de~$\omega$.
Cela démontre que la restriction à~$U_j$ du courant~$S$ est égale à~$S_j$
et conclut la démonstration que les courants forment un faisceau.

Avec les notations qui précèdent, 
l'ordre de~$S$ est~$\leq r$ si les $S_i$ sont tous d'ordre~$\leq r$.
Cela entraîne que 
courants d'ordre~$\leq r$ forment un sous-faisceau du faisceau
des courants.
\end{proof}

\begin{coro}
Si $X$ est localement différentiellement séparé et paracompact, 
les faisceaux~$\mathscr D^{p,q}_X$ sur~$X$
sont des faisceaux fins. 
\end{coro}
\begin{proof}
Puisque ce sont des faisceaux de $\mathscr A^{0,0}_X$-modules,
cela résulte du corollaire~\ref{coro.pl-fin}.
\end{proof}

\subsection{Support d'un courant}
Soit $X$ un espace tropical de niveau~$n$,
localement différentiellement séparé et
soit $T$ un courant de bidimension~$(p,q)$ sur~$X$.
Comme pour toute section d'un faisceau,
le support de~$T$ est défini comme le complémentaire du plus grand
ouvert de~$X$ sur lequel~$T$ est nul. 

Il résulte des définitions que  
pour toute forme~$\omega$ de bidegré~$(p,q)$ sur~$X$,
dont le support est compact et disjoint de~$\supp(T)$,
on a $\langle T,\omega\rangle=0$.

\begin{lemm}
Soit $X$ un espace tropical de niveau~$n$,
localement différentiellement séparé et
soit $T$ un courant de bidimension~$(p,q)$ sur~$X$.
Soit $\Phi$ la famille de supports sur~$X$ constituée des parties
fermées~$F$ telles que $F\cap \supp(T)$ soit compact.

Il existe une unique application,
définie sur l'espace $\mathscr A^{p,q}_{\Phi}(X)$
des formes lisses de bidegré~$(p,q)$ sur~$X$ dont le support
rencontre~$\supp(T)$ en une partie comapcte,
qui applique~$\alpha$ sur $\langle T,\lambda\alpha\rangle$
pour toute fonction lisse à support compact~$\lambda$ sur~$X$
qui vaut~$1$ au voisinage de~$\supp(T)\cap \supp(\alpha)$.

Cette application est une forme linéaire continue qui prolonge~$T$.
\end{lemm}

En particulier, si le support de~$T$ est compact, $\langle T,\alpha\rangle$
est défini pour toute forme lisse de bidegré~$(p,q)$ sur~$X$,
et l'application $\mathscr A^{p,q}(X)\to \R$ ainsi définie
est une forme linéaire continue.

\begin{proof}
D'après le corollaire~\ref{coro.trop-lisse-separe},
il existe une fonction lisse~$\lambda$, à support compact,
et identiquement égale à~$1$ au voisinage de~$\supp(T)\cap\supp(\alpha)$.
Alors $\lambda\alpha$ est une $(p,q)$-forme à support compact
sur~$X$ et l'on peut former l'expression $\langle T,\lambda\alpha\rangle$.

Vérifions qu'elle ne dépend pas du choix de~$\lambda$.
En effet, si $\lambda'$ est un autre choix,
on a $\lambda=\lambda'$ au de~$\supp(T)\cap\supp(\alpha)$,
de sorte que $\supp(\lambda-\lambda')\cap \supp(T)\cap \supp(\alpha)=\emptyset$.
Il en résulte que 
$\supp((\lambda-\lambda')\alpha)\cap \supp(T)=\emptyset$.
On a donc $\langle T, (\lambda-\lambda')\alpha\rangle=0$,
d'où $\langle T,\lambda\alpha\rangle=\langle T,\lambda'\alpha\rangle$.

Cela prouve l'existence et l'unicité d'une application comme indiqué.

Si $\alpha$ est à support compact, on peut supposer
que la fonction~$\lambda$ est égale à~$1$ au voisinage de~$\supp(\alpha)$;
on a alors $\alpha=\lambda\alpha$ ce qui prouve que l'application
construite prolonge~$T$. On la notera toujours~$T$.

Démontrons que ce prolongement~$T$ est linéaire.

Soit $c\in\R$; pour calculer $\langle T,c \alpha\rangle$,
on peut prendre la même fonction~$\lambda$ que celle qui sert
à définir $\langle T,\alpha\rangle$; 
on en déduit que $\langle T,c\alpha\rangle=c \langle T,\alpha\rangle$.

Soit $\alpha$ et $\beta$ deux formes lisses de bidegré~$(p,q)$
dont les supports rencontrent~$\supp(T)$ sur des parties compactes.
On choisit une fonction lisse~$\lambda$, à support compact,
valant~$1$ au voisinage de $(\supp(\alpha)\cup \supp(\beta))\cap\supp(T)$;
On a alors
\[\langle T,\alpha+\beta\rangle=\langle T,\lambda(\alpha+\beta)\rangle
=\langle T,\lambda\alpha\rangle+ \langle T,\lambda\beta\rangle
=\langle T,\alpha\rangle+ \langle T,\beta\rangle,\]
d'où la linéarité.

Démontrons enfin la continuité de~$T$.
Par définition de la topologie sur l'espace~$\mathscr A^{p,q}_{\Phi}(X)$,
il suffit de prouver que pour toute partie fermée~$F$ appartenant à~$\Phi$,
la restriction de~$T$ à l'espace des formes à support dans~$F$ est continue.
On choisir alors $\lambda$ valant~$1$ au voisinage de~$F\cap\supp(T)$.
On a $\langle T,\alpha\rangle=\langle T,\lambda\alpha\rangle$
pour toute forme~$\alpha$ à support dans~$F$,
et la continuité résulte de ce que l'application $\alpha\mapsto \lambda\alpha$
est continue.
\end{proof}

\subsection{Image directe d'un courant}\label{sss.image-directe-courant}
Soit~$X$ un espace tropical de niveau~$n$,
soit~$Y$ un espace tropical de niveau~$m$
et soit~$\phi\colon Y\ra X$
un morphisme d'espaces tropicaux qui est (topologiquement) propre.

 
L'application~$\phi^*$ de~$\mathscr A^{p,q}(X)$
dans~$\mathscr A^{p,q}(Y)$ est linéaire continue
(corollaire~\ref{coro.continuite-pull-back});
comme $\phi$ est topologiquement propre, elle applique
$\mathscr A^{p,q}_\cpct(X)$ dans
dans~$\mathscr A^{p,q}_\cpct(Y)$ et l'application linéaire
induite est continue (\loccit).
Elle induit par dualité une application linéaire
$\phi_*$ de~$\mathscr D_{p,q}(Y)$ dans~$\mathscr D_{p,q}(X)$.

 
Si $T$ est un courant d'ordre~$\leq r$, le courant~$\phi_*T$
est encore d'ordre~$\leq r$.

On a $\phi_* (\mathrm J T)=\mathrm J \phi_*(T)$ pour tout courant~$T$.
En particulier, l'image directe d'un courant symétrique est symétrique.

On a enfin $\supp(\phi_*T)\subset \phi(\supp(T))$.

\medskip

Supposons de plus que les espaces tropicaux~$X$ et $Y$ sont 
localement holomorphiquement séparés, 
mais relâchons l'hypothèse de propreté sur~$\phi$.
Dans ce cas, le courant $\phi_*T$ est défini dès que le support de~$T$
est $\phi$-compact, par la formule $\langle \phi_*T,\alpha\rangle
=\langle T,\phi^*\alpha\rangle$ pour toute forme lisse 
à support compact~$\alpha$ sur~$X$, puisqu'alors l'intersection
des supports de~$T$ et de~$\phi^*\alpha$ est compacte.
On a $\supp(\phi_*T)\subset\phi(\supp(T))$.

En particulier, si le support de~$T$ est compact, 
celui du courant $\phi_*T$ est également compact,
et l'égalité $\langle T,\phi^*\alpha\rangle=\langle \phi_*T,\alpha\rangle$
vaut pour toute forme lisse~$\alpha$ sur~$X$.

\subsection{Prolongement par zéro}\label{ss.prolongement-courant}
Soit~$X$ un espace tropical de niveau~$n$.
Soit $U$ un ouvert de~$X$, notons $j\colon U\hookrightarrow X$
l'immersion canonique,
et soit $T$ un courant sur~$U$ de bidimension~$(p,q)$.
On dit que \emph{$T$ se prolonge par zéro à~$X$}
si, pour toute forme lisse~$\alpha$ sur~$X$, de bidegré~$(p,q)$
et à support compact, le support du courant $T\wedge \alpha|_U$
sur~$U$ est compact.
Dans ce cas, la formule
\[ \langle j_!T,\alpha \rangle = \langle T, \alpha|_U\rangle \]
définit un courant $j_!T$ sur~$X$ qu'on appelle le prolongement
par zéro de~$T$.

Cette application de prolongement est linéaire, 
commute à l'opérateur~$\mathrm J$
et aux opérateurs différentiels~$\di$, $\dc$, $\ddc$, etc.

\begin{lemm}\label{lemm.prolongement-courant-trop}
Supposons de plus que  $d_\trop(x)<\max(p,q)$
pour tout $x\in X\setminus U$.

Tout courant de bidimension~$(p,q)$ se prolonge par zéro à~$X$.

L'application de restriction de $\mathscr D^{p,q}(X)$ dans~$\mathscr D^{p,q}(U)$
est un isomorphisme, l'application réciproque est donnée 
par $T\mapsto j_!T$.
\end{lemm}
Dans le cas des espaces de Berkovich,
l'exemple typique sera le cas où $q=p$ et où $U$ est le complémentaire
d'un fermé de Zariski de dimension~$<p$.
\begin{proof}
Soit $T$ un courant sur~$U$.
Pour toute $(p,q)$-forme lisse $\alpha$ à support compact sur~$X$,
le support de~$\alpha$ est contenu dans~$U$
(lemme~\ref{lemm.support-comp-ferme}),
donc le support de $T\wedge \alpha|_U$ est contenu dans~$U$. 
Par suite,  le courant~$T$ se prolonge par zéro à~$X$.
Plus précisément, l'application $\alpha\mapsto \alpha|_U$ est 
un isomorphisme d'espace vectoriels topologiques
de $\mathscr A_\cpct^{n-p,n-q}(X)$ sur $\mathscr A_\cpct^{n-p,n-q}(U)$
(lemme~\ref{prop.continuite-push-U'}), de sorte que 
le courant~$j_!T$ est le seul courant de bidimension~$(p,q)$ sur~$X$
qui prolonge~$T$.
\end{proof}

\section{Courants positifs}

Soit $X$ un espace tropical de niveau~$n$.
\begin{defi}\label{defi.courant-positif}
Soit $p$ un entier.
Un courant~$S$ de bidimension~$(p,p)$ est dit 
\emph{faiblement positif}, respectivement \emph{positif},
respectivement \emph{fortement positif}, 
s'il est symétrique
et si pour toute forme
$\alpha\in\mathscr A^{p,p}_\cpct(X)$
qui est fortement positive, respectivement positive, respectivement
faiblement positive,
on a $\langle S,\alpha\rangle\geq 0$.\index{courant positif}
\end{defi}

\subsection{}
Les trois notions  de positivité pour
les formes de type~$(0,0)$ et $(n,n)$ coïncident;
il en est donc de même pour les courants.
Comme les différentes notions de positivité
coïncident pour les formes symétriques de type $(1,1)$ et $(n-1,n-1)$,
il en est de même pour les courants symétriques.

\begin{rema}
Soit $T$ un courant tel que $T$ et $-T$ sont positifs; démontrons
que $T=0$. Soit $\omega$ une  forme lisse à support compact sur~$X$;
prouvons que $\langle T,\omega\rangle=0$.
Si $\omega$ est antisymétrique, cela découle de la symétrie de~$T$;
supposons donc que $\omega$ est symétrique.
Il existe aussi des formes positives $\alpha_1$ et $\alpha_2$,
lisses à support compact telles que $\omega=\alpha_1-\alpha_2$
(proposition~\ref{prop.trop-diff-positives}). Par linéarité, on peut 
donc supposer que  $\omega$ est positive.
Mais alors, $\langle T,\omega\rangle$ est à la fois positif
($T$ est positif) et négatif ($-T$ est positif), donc nul.
\end{rema}

\begin{lemm}\label{lemm.courants-positif-faisceau}
Supposons que l'espace tropical~$X$ soit 
localement différentiellement séparé.
Les courants de bidimension~$(p,p)$ positifs
(\resp faiblement positifs, \resp fortement positifs)
forment un sous-faisceau du faisceau~$\mathscr D_{p,p}(X)$
des courants de bidimension~$(p,p)$.
\end{lemm}
\begin{proof}
Soit $S$ un courant de bidimension~$(p,p)$ sur~$X$.

Supposons que $S$ soit positif.
Soit $U$ un ouvert de~$X$. Pour toute forme~$\omega$ sur~$U$,
de type~$(p,p)$, positive et à support compact,
le prolongement par~$0$ de~$\omega$ est une forme positive à support
compact sur~$X$. Il en résulte que $S|_U$ est positif.

Soit maintenant~$(U_i)$ un recouvrement ouvert de~$X$ 
tel que $S|_{U_i}$ soit positif pour tout~$i$; démontrons
que $S$ est positif.
Soit $\omega$ une forme de type~$(p,p)$ sur~$X$,
positive et à support compact. Reprenons la preuve 
que les faisceaux forment un courant (proposition~\ref{prop.courants-faisceau}).
Il existe une famille finie $(\phi_i)$ 
de fonctions lisses, positives et à support compact sur~$X$
telle que pour tout~$i$, le support de~$\phi_i$ soit contenu dans~$U_i$,
et telle que $\sum \phi_i$ 
soit égale à~$1$ au voisinage du support de~$\omega$.
On a de plus 
 $\langle S,\omega\rangle=\sum \langle S,\phi_i\omega\rangle\geq 0$.
Pour tout~$i$, on a a $\langle S,\phi_i\omega\rangle\geq 0$,
car $\phi_i\omega$ est une forme positive dont le support est
compact et contenu dans~$U_i$.
Cela démontre que $S$ est positif.

Appliqué à chaque ouvert de~$X$, cet argument démontre que
les courants positifs de bidimension~$(p,p)$ forment un 
sous-faisceau du faisceau des courants de bidimension~$(p,p)$.
Les cas des courants faiblement positifs et fortement positifs
se démontrent de même.
\end{proof}

\begin{lemm}\label{rema.formes-test}
Supposons que $X$ soit localement différentiellement séparé.
Soit $T$ un courant sur~$X$.
On suppose que pour tout ouvert~$U$ de~$X$,
toute forme tropicale~$\omega$ sur~$U$ qui est positive
(\resp faiblement positive, \resp fortement positive)
et toute fonction lisse positive~$\rho$ 
dont le support est compact et contenu dans~$U$,
on a $\langle T,\rho\omega\rangle\geq 0$.
Alors $T$ est positif 
(\resp fortement positif, \resp faiblement positif).
\end{lemm}
\begin{proof}
On traite le cas des formes positives, les deux autres sont analogues.
Soit $\omega$ une forme lisse positive à support compact sur~$X$;
démontrons que $\langle T,\omega\rangle\geq 0$.
Pour tout~$x\in X$, soit $U_x$ un voisinage ouvert de~$X$
sur lequel $\omega$ est tropicale.
Soit $(\lambda_x)_{x\in X}$ une famille à support fini de fonctions lisses
positives
telle que, pour tout~$x\in X$, le support de~$\lambda_x$ soit compact
et contenu dans~$U_x$, et telle que $\sum\lambda_x$ soit
égale à~$1$ au voisinage du support de~$\omega$
(corollaire~\ref{prop.trop-partition.1.lisse-compact}).
On a alors $\omega = \sum_{x\in X} \lambda_x \omega$.
Par hypothèse, $\langle T,\lambda_x\omega\rangle\geq 0$ pour tout~$x$.
Donc $\langle T,\omega\rangle\geq 0$.
\end{proof}

\subsection{}
L'image directe d'un courant positif (\resp faiblement positif,
fortement positif) par un morphisme topologiquement propre
d'espaces tropicaux est de même type.
Cela résulte du fait que la même propriété vaut
pour l'image inverse des formes (\S\ref{rema.image.inverse}).

\begin{lemm}\label{lemm.produit-courant-positif-forme}
Soit $\alpha$ une forme lisse symétrique sur~$X$
et soit $T$ un courant symétrique sur~$X$.

\begin{enumerate}
\item Supposons que la forme $\alpha$ soit fortement positive.
Alors, si $T$ est positif (\resp faiblement positif, \resp fortement positif), il en est de même de $T\wedge\alpha$.

\item Supposons que la forme $\alpha$ soit positive.
Si $T$ est positif, le courant $T\wedge\alpha$ est positif.

\item Supposons que la forme $\alpha$ soit faiblement positive.
Si $T$ est fortement positif, le courant $T\wedge\alpha$ est faiblement
positif.
\end{enumerate}
\end{lemm}
\begin{proof}
Ces assertions découlent
des propriétés de positivité d'un produit de deux formes 
(prop.~\ref{prop.produit-formes-positives-berk}).
\end{proof}

\begin{prop}\label{prop.mesure-courant-positif}
Soit $X$ un espace tropical de niveau~$n$,
localement différentiellement séparé.

\begin{enumerate}
\item Toute mesure de Radon positive~$\mu$ sur~$X$ induit, par restriction
à~$\mathscr A_\cpct(X)$, un courant positif~$[\mu]$ d'ordre~$0$.
\item L'application $\mu\mapsto [\mu]$ ainsi définie est
une bijection continue de l'espace des mesures de Radon positives sur~$X$
sur l'espace des courants positifs de bidimension~$(0,0)$ sur~$X$.
\end{enumerate}
\end{prop}
\begin{proof}
\begin{enumerate}
\item Soit $\mu$ une mesure de Radon positive sur~$X$;
démontrons que la forme linéaire~$[\mu]$ est un courant;
ce sera alors un courant positif.
Soit $K$ une partie compacte de~$X$ et soit $\mathscr D$ un jeu de cartes
couvrant sur~$X$. Il suffit de démontrer que la restriction de~$[\mu]$
à l'espace $\mathscr A^{0,0}_{\mathscr D,K}(X)$ est continue pour
sa topologie~$\mathscr C^0$.
Soit $\norm\mu_K$ la norme de~$\mu$ sur~$K$: pour toute fonction
continue à support compact sur~$K$, on a $\abs{\mu(f)}\leq \norm\mu_K \norm f_K$. Cette inégalité  vaut a fortiori pour toute fonction de~$\mathscr A_{\mathscr D,K}(X)$ et entraîne la continuité voulue.

\item Cette implication est injective car l'espace des
fonctions lisses à support compact est dense dans l'espace
des fonctions continues à support compact
(corollaire~\ref{coro.trop-sw}).

Il reste à prouver qu'elle est surjective. Soit $S$ un courant
positif de bidimension~$(0,0)$ sur~$X$.
Soit $K$ une partie compacte de~$X$.
Soit $\lambda$ une fonction lisse et positive, à support compact sur~$X$,
qui vaut~$1$ au voisinage de~$K$.
Pour toute fonction lisse~$f$ sur~$X$ dont le support est contenu
dans~$K$, on a $-\norm f_K \lambda \leq f \leq \norm f_K \lambda$.
Comme $S$ est positif, il en résulte l'inégalité
\[ \abs{S(f)} \leq \norm f_K S(\lambda). \]
Par suite, $S$ est continu lorsque l'espace des fonctions lisses
à support compact est muni de la topologie de la convergence compacte.
Par densité, $S$ s'étend en une forme linéaire positive sur~$\mathscr C_c(X)$,
c'est-à-dire en une mesure de Radon positive.
\qedhere\end{enumerate}
\end{proof}

La démonstration précédente n'a utilisé que la positivité de la forme linéaire~$S$ et pas sa continuité. La proposition 
suivante montre comment ce fait se généralise en tout degré.

\begin{prop}\label{prop.positif-courant}
Soit $X$ un espace tropical de niveau~$n$,
localement différentiellement séparé,
soit~$p$ un entier et soit~$S$ une application linéaire 
de~$\mathscr A^{p,p}_\cpct(X)$ dans~$\R$
telle que $S(\mathrm J\omega)=S(\omega)$
pour tout $\omega\in\mathscr A^{p,p}_\cpct(X)$.
Si $S(\omega)\geq 0$ pour  toute forme fortement positive~$\omega$,
alors $S$ est un courant faiblement positif, d'ordre~$0$.
\end{prop}
\begin{proof}
Il s'agit de prouver la continuité de~$S$ lorsque $\mathscr A^{p,p}_\cpct(X)$
est muni de la topologie~$\mathscr C^0$.
D'après le lemme~\ref{lemm.continuite-courant},
il suffit de démontrer l'énoncé suivant:
soit $f$ une carte tropicale définie sur une partie ouverte~$U$ de~$X$, 
soit $\rho$ une fonction lisse, positive et à support compact sur~$U$,
la forme linéaire $\omega\mapsto S(\rho\omega)$ sur~$\mathscr A^{p,p}_f(U)$
est continue pour la topologie~$\mathscr C^0$.

Soit $W$ un domaine compact de~$U$ qui contient le support de~$\rho$.
Notons $E$ l'espace affine codomaine de~$f$ et munissons-le
d'une norme euclidienne; notons $\eta=\ddc\norm x^2$ la $(1,1)$-forme
fortement positive sur~$E$.
Soit $P$ l'espace paralinéaire compact~$f(W)$;
considérons une norme~$\mathscr C^0$ sur l'espace~$\mathscr A^{p,p}(P)$.
On déduit de la proposition~\ref{prop.harvey} qu'il existe
un nombre réel~$c>0$ tel que la forme 
$\alpha+c \norm\alpha \eta^p$ est fortement positive,
pour tout $\alpha\in\mathscr A^{p,p}(P)$.

Pour toute $\omega\in\mathscr A^{p,p}_f(U)$;
il existe une unique forme $\alpha\in\mathscr A^{p,p}(P)$
telle que $\omega|_W=f^*\alpha$, et l'application $\omega\mapsto
\norm{\omega}_{P,0}$ est une semi-norme~$\mathscr C^0$ sur~$\mathscr A_f(U)$.
On la note $\norm{\cdot}_{W,0}$.

Soit $\omega\in\mathscr A^{p,p}_f(U)$;
il existe une unique forme $\alpha\in\mathscr A^{p,p}(P)$
telle que $\omega|_W=f^*\alpha$, si bien que 
$ \omega + c \norm\alpha f^*\eta^p$ est fortement positive.
Il en est de même de $\rho(\omega+c\norm\alpha_{P,0} f^*\eta^p)$.
Par l'hypothèse sur~$S$, on a donc 
\[ S(\rho\omega) + c \norm\omega_{W,0} S(\rho f^*\eta^p) \geq 0. \]
En appliquant ce raisonnement à~$-\omega$, on en déduit l'inégalité
\[ \abs{ S(\rho\omega)}\leq c S(\rho f^*\eta^p) \norm \omega_{W,0}. \]
Cela démontre la continuité requise.

Ainsi, $S$ est un courant.
La condition $S(\mathrm J\omega)=S(\omega)$ implique alors
que c'est un courant symétrique,
et la condition de positivité $S(\omega)\geq 0$
pour $\omega$ positive entraîne que c'est un courant faiblement positif.
\end{proof}

\begin{coro}\label{coro.courant-L1loc}
Soit $u$ une fonction localement intégrable sur~$X$.
La forme linéaire $\omega\mapsto \int_X u\omega$ sur~$\mathscr A^{n,n}_\cpct(X)$
définit un courant d'ordre~$0$ 
et de bidimension~$(n,n)$ sur~$X$, fortement positif si $u$ est positif.
\end{coro}
On note $[u]$ ce courant.
\begin{proof}
Les fonctions $u_+=\sup(u,0)$ et $u_-=\sup(-u,0)$ sont
localement intégrables.
Les formes linéaires qu'elles définissent sont 
fortement positives, donc sont des courants d'ordre~$0$.
Comme on a $u=u_+-u_-$, la forme linéaire donnée est la différence
de ces deux courants, donc est un courant.
\end{proof}

\subsection{}\label{ss.courant-mesure}
Soit $X$ un espace tropical de niveau~$n$,
localement différentiellement séparé.
Soit $p$ un entier naturel et soit $T$ un courant positif
(\resp faiblement positif, \resp fortement positif)
de bidimension~$(p,p)$.

Soit $\omega$ une forme positive (\resp fortement positive,
\resp faiblement positive) de bidegré~$(p,p)$.
Le courant $T\wedge\omega$ est positif, de bidimension~$(0,0)$;
d'après la proposition~\ref{prop.mesure-courant-positif},
il est donc associé à une unique mesure~$\mu^T_\omega$ sur~$X$.
Cette mesure applique une fonction lisse à suppport compact~$u$
sur $\langle uT, \omega\rangle$.
Lorsque $\omega$ est à support compact, cette mesure est à support compact.
Lorsque $\omega$ est positive, cette mesure est positive.

\begin{exem}\label{coro.courant-positif}
Soit $X$ un espace tropical de niveau~$n$,
localement différentiellement séparé.
Soit $p$ un entier naturel et soit $T$ un courant positif
(\resp faiblement positif, \resp fortement positif)
de bidimension~$(p,p)$.
Soit $u$ une fonction borélienne et localement bornée sur~$X$.
L'application $\omega\mapsto \mu^T_\omega(u)$ est une forme
linéaire sur $\mathscr A^{p,p}_\cpct(X)$ que nous noterons~$uT$.

Lorsque $u$ est positive, cette forme linéaire~$uT$ est positive,
donc est un courant positif (\resp faiblement positif,
\resp fortement positif) sur~$X$.

Lorsque $u$ est une fonction borélienne et localement bornée
arbitraire, on l'écrit $u=u^+-u^-$, où $u^+=\sup(u,0)$ et $u^-=\sup(-u,0)$,
comme différence de deux fonctions positives boréliennes localement
bornées sur~$X$. Comme $uT=u^+T-u^-T$, c'est un courant sur~$X$, d'ordre~$0$.

Cette notation est compatible avec celle déjà introduite lorsque
$u$ est une fonction lisse puisque l'on a alors
$\langle uT, \omega\rangle=\langle T,u\omega\rangle=\mu^T_\omega(u)$.
\end{exem}

\begin{exem}\label{prop.courant.positif-mesure}
Soit $(\mu_n)$ une suite de mesures de Radon positives 
et soit $\mu$ une mesure de Radon positive sur~$X$.
Si $[\mu_n]\to [\mu]$ au sens des courants, alors $\mu_n\to\mu$ au sens
des mesures. Cela résulte de la densité des fonctions lisses 
à support compact dans les fonctions continues à support compact 
(corollaire~\ref{coro.trop-sw}).
\end{exem}

\begin{exem}
Il n'est pas vrai que tout courant d'ordre~$0$ se prolonge en une mesure.
Considérons l'étoile~$X_B$ associée à un ensemble infini~$B$
(\cf\S\ref{defi-etoile}, dont nous reprenons les notations).
Soit $(c_b)_{b\in B}$ une famille de nombres réels.
Pour toute fonction lisse~$f$ sur~$X$, posons
\[ \langle T,f\rangle = \sum_{b\in B} c_b \big( f_b(0) -f_b(-\infty)\big). \]
Cela définit un courant sur~$X$, de bidimension~$(0,0)$,
mais il ne provient d'une mesure que lorsque 
la famille $(c_b)$ est sommable. Cette mesure est alors égale à
\[ \sum_{b\in B} c_b \delta_\eta - \sum_{b\in B} c_b \delta_{p(b,-\infty)}. \]
\end{exem}

\section{Courants d'intégration}\label{ss.courants.integration}

\subsection{}
Soit~$X$ un espace tropical de niveau~$n$.
Pour toute~$(n,n)$-forme lisse~$\alpha$ à support compact sur~$X$,
posons alors~$\langle \delta_X,\alpha\rangle
=\int_X\alpha$,
où l'intégrale d'une~$(n,n)$-forme  lisse à  support compact
sur~$X$ a été définie en~(\ref{defi.integrale.n,n}).
On définit ainsi une forme linéaire
sur l'espace~$\mathscr A^{n,n}_\cpct(X)$. 
Elle est positive ; d'après la proposition~\ref{prop.positif-courant},
c'est donc un courant sur~$X$, 
qu'on appelle le \emph{courant d'intégration sur~$X$}
et on le note~$\delta_X$.\index{courant d'intégration}
Il est de bidimension~$(n,n)$ et d'ordre~$0$.
Une forme antisymétrique de type~$(n,n)$ sur~$X$ étant nulle, 
donc son intégrale aussi, le courant~$\delta_X$ est symétrique. 

\subsection{}
On définit de manière analogue le courant d'intégration « sur le bord » de~$X$,
noté~$\delta^\partial_{X}$, par la formule
$\langle \alpha,\delta^\partial_X\rangle=\int^\partial_X \alpha$,
pour toute~$(n-1,n)$-forme  à support compact~$\alpha$ sur~$X$.

Cette forme linéaire est continue pour la topologie~$\mathscr C^0$.
Compte tenu de la définition
de~$\int_X\alpha$, cela découle en effet de la continuité 
pour la topologie~$\mathscr C^0$
de l'intégrale de bord 
d'une forme sur un polytope calibré dans un espace affine.

Il est de bidimension~$(n-1,n)$ et d'ordre~$0$; son support
est contenu dans le bord de~$X$.

\subsection{}
Soit $Y$ un espace tropical de niveau~$m$
et soit $\phi\colon Y\to X$ un morphisme topologiquement propre
d'espaces tropicaux. 
Un cas particulier important est celui d'une partie squelettique
de~$X$ munie de son $n$-calibrage canonique,
plus généralement d'une partie squelettique de dimension~$\leq d$ quelconque
munie d'un $d$-calibrage.

Les constructions précédentes fournissent des courants
$\phi_*\delta_Y$ et $\phi_*\delta^\partial_{Y}$
sur~$X$ de bidimensions~$(m,m)$ et $(m-1,m)$ respectivement.

Le courant~$\phi_*\delta_Y$ est fortement positif.

 
\subsection{}\label{sss.[omega]}\label{ss.positivite-forme-courant}
Soit $\omega$ une G-forme de type~$(p,q)$ 
à coefficients tropicalement boréliens localement bornés sur~$X$. 
On définit un courant $[\omega]$ de bidimension~$(p,q)$ et d'ordre~$0$
sur~$X$ par la formule
$\alpha\mapsto \int_X \omega\wedge\alpha$ 
définie par l'intégrale des G-formes.
Lorsque $\omega$ est une forme lisse, on a 
$[\omega]=\delta_X\wedge\omega$.

Si $\omega$ est positive (\resp faiblement positive, \resp fortement positive),
il en va de même du courant~$[\omega]$.
Cela découle de la définition  d'un courant positif,
des propriétés de stabilité de la positivité par produit
(prop.~\ref{prop.produit-formes-positives-berk}),
et de ce que l'intégrale d'une $(n,n)$-forme positive est positive.

On définit aussi le courant $[\omega]^\partial$
par la formule $\alpha\mapsto \int^\partial_X \omega\wedge\alpha$
défini à l'aide de l'intégrale de bord des G-formes.
Il est d'ordre~$0$ et de bidimension $(n-p-1,n-q)$.
Lorsque $\omega$ est une forme lisse,
on a $[\omega]^\partial=\delta^\partial_X\wedge\omega$.

Il découle de la définition de la topologie des espaces
de courants que les applications $\omega\mapsto [\omega]$
et $\omega\mapsto [\omega]^\partial$ sont continues.

 
\subsection{}
Nous allons maintenant généraliser  cette construction aux 
« G-formes » qui sont combinaisons linéaires de G-formes usuelles
à coefficients dans un anneau de fonctions non nécessairement tropicales.

Soit $V$ un domaine de~$X$ et soit $U$ un voisinage ouvert de~$V$
dans~$X$ dans lequel $V$ est fermé. 
Soit $u\in\mathrm L^1_\loc(V)$, $\omega \in\mathscr A^{p,q}_{X_\groth}(V)$
et $\alpha\in\mathscr A^{n-p,n-q}_\cpct(U)$; la G-forme $\omega\wedge\alpha|_V$
est à support compact dans~$V$, de sorte que $u$ est intégrable
pour la mesure associée. L'application $\alpha\mapsto
\int_V u(\omega \wedge \alpha)_\Rad$ est un courant d'ordre~$0$ sur~$U$.
La section correspondante de $\pi^* \mathscr D_{p,q} (V)$
ne dépend pas du choix de~$U$.
On définit ainsi un morphisme de G-faisceaux 
\[ \mathrm L^1_{\loc,X}\otimes \mathscr A^{p,q}_{X_\groth} \to 
\pi^*(\mathscr D_{p,q})_X . \]
Par adjonction, on en déduit un morphisme de faisceaux
\[ \pi_*(\mathrm L^1_{\loc,X}\otimes \mathscr A^{p,q}_{X_\groth}) \to 
\pi_*\pi^* (\mathscr D_{p,q})_X . \]
Comme l'unité $\id\to\pi_*\pi^*$ est un isomorphisme
de foncteurs, on en déduit un morphisme de faisceaux
\[ \pi_*(\mathrm L^1_{\loc,X}\otimes \mathscr A^{p,q}_{X_\groth}) \to 
 (\mathscr D_{p,q})_X . \]

L'image d'une section~$\omega$ 
sera notée~$[\omega]$. 
La notation est cohérente avec le cas des G-formes lisses;
via l'isomorphisme de la proposition~\ref{lemm.otimes-coeff},
elle est aussi cohérente avec les cas des G-formes à coefficients
tropicalement boréliens localement bornés.
Ce courant est d'ordre~$0$ et de bidimension~$(n-p,n-q)$.

Par composition avec l'homomorphisme canonique
\[ \pi_*(\mathscr L^1_{\loc,X}\otimes \mathscr A^{p,q}_{X_\groth}) \to 
 \pi_*(\mathrm L^1_{\loc,X}\otimes \mathscr A^{p,q}_{X_\groth}), \]
on en déduit un 
homomorphisme
\[ \pi_*(\mathscr L^1_{\loc,X}\otimes \mathscr A^{p,q}_{X_\groth}) \to 
(\mathscr D_{p,q})_X  \]
qu'on note encore $\omega\mapsto[\omega]$.

Pour toute forme~$\alpha$ de type~$(n-p,n-q)$, lisse à support compact,
on a 
\[ \langle\mathrm J[\omega],\alpha\rangle
= \langle [\omega],\mathrm J^{-1}\alpha\rangle
=\int_X \omega\wedge \mathrm J^{-1}\alpha
= \int_X \mathrm J^{-1} (\mathrm J \omega \wedge \alpha)
= \int_X \mathrm J\omega\wedge\alpha
= \langle [\mathrm J\omega],\alpha\rangle.\]
Ainsi,  $\mathrm J[\omega]= [\mathrm J\omega]$.

On laisse au lecteur le soin d'écrire la variante de cette construction
pour les intégrales de bord.

\begin{prop}
\label{prop.courant-restriction}
Soit $\omega$ une G-forme (à coefficients tropicalement
boréliens et localement bornés) sur~$X$.

\begin{enumerate}
\item
Pour que le courant~$[\omega]$ soit nul (respectivement positif,
respectivement faiblement positif, respectivement fortement positif),
il faut et il suffit
que pour toute partie squelettique~$\Sigma$ de~$X$,
la forme~$\omega|_\Sigma$ soit nulle (respectivement positive,
respectivement\dots)
en dehors d'une partie $n$-négligeable.

\item
Si $\omega$ est à coefficients tropicalement continus,
le courant~$[\omega]$ est nul (respectivement positif,
respectivement faiblement positif, respectivement fortement positif)
si et seulement si $\omega|_\Sigma$ est nulle
(respectivement positive, respectivement\dots)
pour toute partie squelettique~$\Sigma$ de~$X$ qui
est purement de dimension~$n$.
\end{enumerate}
\end{prop}
\begin{proof}
La condition est évidemment suffisante; prouvons qu'elle est nécessaire.
Soit $\omega$ une G-forme de type~$(p,q)$ 
dont le courant associé~$[\omega]$ est nul (respectivement positif,
etc.).
Soit $\Sigma$ une partie squelettique;
prouvons que la forme~$\omega|_\Sigma$ est 
nulle (respectivement positive, etc.)
en dehors d'une partie localement $n$-négligeable de~$\Sigma$.
Soit $\Sigma'$ l'ensemble des points de~$\Sigma$ au voisinage
desquels $\omega|_\Sigma$ a la propriété voulue;
démontrons que $\Sigma\setminus\Sigma'$ est $n$-négligeable.

C'est une propriété locale, ce qui permet de supposer l'espace paracompact.
Elle se teste alors sur un G-recouvrement localement fini
ce qui permet de supposer qu'il existe un domaine compact~$V$
de~$X$ dont~$\Sigma\cap V$ est 
une partie fidèlement paralinéaire fermée sur lequel
$\omega$ est tropicale.
Quitte à restreindre encore~$V$,
il existe donc une application paralinéaire~$f$ sur~$V$ 
qui G-tropicalise~$\omega|_V$ et atteste
que $\Sigma\cap V$ est fidèlement paralinéaire
(définition~\ref{defi.fidelement-pl}).
Quitte à introduire une décomposition cellulaire de~$\Sigma \cap V$
en $f$-cellules et à projeter, on se ramène au cas
où $f$ est à valeurs dans~$\R^n$ et où $\Sigma \cap V$ est une $f$-cellule 
purement de dimension~$n$.
Par définition d'un espace tropical, $f(\partial(V))$
est $n$-négligeable. Quitte à  raffiner la décomposition cellulaire
choisie, on suppose que $C=f(\Sigma \cap V)$ est une cellule forte de~$\R^n$
dont le bord contient~$f(\partial(V))$.

Soit $\alpha$ une $(p,q)$-forme sur~$C$ telle que $\omega|_\Sigma=f^*\alpha$.
Si le courant~$[\omega]$ est nul,
justifions que pour toute $(n-p,n-q)$-forme~$\beta$,
lisse, à support compact, sur~$\mathring C$,
on a $\int_C \alpha\wedge\beta=0$;
si ce courant est positif (resp. faiblement positif, resp. fortement positif),
justifions que pour toute forme~$(n-p,n-q)$-forme~$\beta$,
lisse, positive (resp. fortement positive, resp. faiblement positive),
à support compact, sur~$\mathring C$, 
on a $\int_C \alpha\wedge\beta\geq0$.

Considérons la forme $\eta=f^*\beta$ sur~$V$; par construction,
c'est une $(n-p,n-q)$-forme  lisse sur~$V$ 
et son support est contenu dans~$\Int(V)$ qui est lui-même
contenu dans $\Int(V/X)$ ; elle se prolonge donc par zéro
en une $(n-p,n-q)$-forme lisse sur~$X$, 
que l'on note encore~$\eta$ qui est à support compact.
Si $\beta$ est positive (resp. fortement positive, resp. faiblement positive),
la forme~$\eta$ l'est également.

Si le courant~$[\omega]$ est nul, on a $\int_X \omega\wedge\eta=0$.
D'après la proposition~\ref{proposition.formule-au-but}, 
il vient donc $\int_C \alpha\wedge\beta=0$.
Prenons pour~$\beta$ une forme du type $u \beta_0$,
où $\beta_0$ est une $(n-p,n-q)$-superforme « constante » sur~$\R^n$
et $u$ est une fonction lisse à support compact sur~$C$.
En faisant varier~$u$,
on déduit de ce qui précède
que l'ensemble des $x\in C$ tels que $\alpha(x)\wedge\beta_0$
soit non nulle est de mesure nulle dans~$C$.
(On utilise le fait si $\lambda$ est une fonction borélienne bornée sur~$C$
à valeurs réelles telle que $\int_C u\lambda = 0$ pour toute
fonction continue à support compact~$u$ sur~$C$, alors l'ensemble
des~$x\in C$ tels que $\lambda(x)\neq 0$ est de mesure nulle.)
Faisons maintenant varier~$\beta_0$ dans une base des
superformes sur~$\R^n$; on obtient un ensemble
négligeable~$N$ de~$C$ tel que, pour tout $x\in C\setminus N$,
$\alpha(x)\wedge\gamma=0$, pour toute superforme~$\gamma$
de type~$(n-p,n-q)$.
Par dualité, $\alpha(x)=0$ pour tout $x\in C\setminus N$,
d'où le résultat.

Si le courant~$[\omega]$ est positif, on a $\int_X\omega\wedge\eta\geq0$
puis, comme ci-dessus, $\int_C\alpha\wedge\beta\geq 0$.
On déduit de même que pour toute $(n-p,n-q)$-superforme « constante »
$\beta_0$ sur~$\R^n$  qui est positive (resp. fortement positive, resp.
faiblement positive), l'ensemble des $x\in C$ 
tels que $\alpha(x)\wedge\beta_0$ ne soit pas positive
est de mesure nulle dans~$C$.
(On utilise le fait si $\lambda$ est une fonction borélienne bornée sur~$C$
à valeurs réelles telle que $\int_C u\lambda \geq 0$ pour toute
fonction continue positive et à support compact~$u$ sur~$C$, alors l'ensemble
des~$x\in C$ tels que $\lambda(x)< 0$ est de mesure nulle.)
Faisons maintenant varier~$\beta_0$ dans une partie dénombrable dense
du cône des formes positives (resp. fortement positives,
resp. faiblement positives). On obtient un ensemble
négligeable~$N$ de~$C$ tel que, pour tout $x\in C\setminus N$,
$\alpha(x)\wedge\gamma$ est positive, pour toute superforme~$\gamma$
qui est positive (resp\dots).
Par dualité, $\alpha(x)$ est positive (resp. faiblement
positive, resp. fortement positive) pour tout $x\in C\setminus N$,
d'où le résultat.

Si le support de~$\omega|_\Sigma$ est négligeable, son intérieur est vide.
Si $\omega|_\Sigma$ est à coefficients continus,
son support est l'adhérence de son intérieur,
la dernière assertion en découle immédiatement.
\end{proof}

\begin{coro}
Soit $X$ un espace tropical de niveau~$n$,
soit $\omega$ une G-forme de type~$(n,n)$ sur~$X$, 
à coefficients tropicalement continus et à support compact.
Si $\omega$ est positive et non nulle, on a $\int_X \omega>0$.
\end{coro}
\begin{proof}
Soit $\omega$ une forme positive à support compact telle que $\int_X\omega=0$.
Soit $K$ le support de~$\omega$. Comme $\omega$ est positive, 
les formes $(u-\inf_K (u))\omega$ et $(\sup_K (u)-u)\omega$ sont positives,
leurs intégrales le sont aussi, 
d'où 
\[ 0= \inf_K(u) \int_X \omega\leq \int_X u \omega \leq \sup_K (u) \int_X\omega = 0  \]
et l'égalité $\int_X u\omega=0$.
En particulier, le courant~$[\omega]$ est nul.
Alors, $\omega|_\Sigma=0$ pour toute partie
squelettique~$\Sigma$.
En appliquant cette égalité à un support fort de~$\omega$
(proposition~\ref{lemm.squelette-support-fort}),
on en déduit que $\omega=0$. Cela démontre le corollaire.
\end{proof}
 
\begin{coro}
\label{coro.injectivite-n,n}
Soit $X$ un espace tropical de niveau~$n$.
Soit $p$ et $q$ des entiers tels que $\sup(p,q)=n$.
Notons
$\pi\colon X_\groth\to X$ le morphisme de sites naturel.
L'homomorphisme de faisceaux 
$\pi_*\mathscr C_{\groth,X}^{p,q}\to (\mathscr D_{p,q})_X$
est
injectif.
\end{coro}
\begin{proof}
Traitons le cas où $q=n$, l'autre cas étant analogue.
Soit $U$ un ouvert de~$X$ et soit $\omega$ une forme de type~$(p,n)$
sur~$X$ telle que $[\omega]=0$; prouvons que $\omega=0$.
On peut raisonner localement sur~$U$ et donc supposer que $U$ est paracompact.
Dans ce cas, $\omega$ possède un support fort, $\Sigma$
(proposition~\ref{lemm.squelette-support-fort}),
et le support de~$\omega|_\Sigma$ est égal au support de~$\omega$
(proposition~\ref{prop.restr-support-fort}).
Puisque $[\omega]=0$, la proposition~\ref{prop.courant-restriction}
entraîne que $\omega|_\Sigma=0$.
On a donc $\supp(\omega)=\emptyset$, c'est-à-dire $\omega=0$.
\end{proof}

\begin{coro}\label{coro.injectivite-p,q}
Soit $X$ un espace tropical de niveau~$n$, \emph{très riche}.
Pour tout couple~$(p,q)$ d'entiers naturels,
l'homomorphisme de faisceaux 
$\pi_*\mathscr C_{\groth,X}^{p,q}\to (\mathscr D_{p,q})_X$
est injectif.
\end{coro}
\begin{proof}
Soit $U$ un ouvert de~$X$ et soit 
$\omega$ une $(p,q)$-forme à coefficients tropicalement continus sur~$U$.
Supposons que $\omega\neq0$. Le corollaire~\ref{coro.tres-riche-support}
affirme l'existence d'une partie squelettique~$\Sigma$ de~$U$,
purement de dimension~$n$, telle que $\omega|_\Sigma$ soit non nulle.
La proposition~\ref{prop.courant-restriction} entraîne alors
que le courant~$[\omega]$ sur~$U$ n'est pas nul.
\end{proof}

%
%
 
\begin{coro}\label{coro.positivite-courant-forme}
Soit $X$ un espace tropical de niveau~$n$, très riche.
Soit $p$ un entier naturel et soit $\omega$ une forme de type~$(p,p)$ sur~$X$,
symétrique, à coefficients tropicalement  continus.
Pour que le courant~$[\omega]$ soit positif (resp. faiblement positif, resp.
fortement positif), il faut et il suffit
qu'il en soit de même de la forme~$\omega$.
\end{coro}
\begin{proof}
Si $\omega$ est positive, on a déjà expliqué (\S\ref{ss.positivite-forme-courant}) que le courant associé est positif, pour les trois types de positivité.
Établissons la réciproque par contraposition.
Supposons que $\omega$ ne soit pas positive.
Le corollaire~\ref{coro.tres-riche-support}
affirme l'existence d'une partie squelettique~$\Sigma$ de~$X$,
de dimension~$n$, telle que $\omega|_\Sigma$ n'est pas positive.
La proposition~\ref{prop.courant-restriction} entraîne alors
que le courant~$[\omega]$ sur~$X$ n'est pas positif.
Le cas des deux autres types de positivité se démontre de même.
\end{proof}

\begin{coro}\label{coro.restriction-interieur}
Soit $X$ un espace tropical de niveau~$n$, très riche.
Soit $\omega$ une G-forme de type~$(p,q)$ sur~$X$
à coefficients tropicalement continus.

\begin{enumerate}
\item Si $\omega|_{\Int(X)}=0$, alors $\omega=0$.
\item Si $p=q$ et si $\omega|_{\Int(X)}$ est positive,
alors $\omega$ est positive.
\end{enumerate}
\end{coro}
\begin{proof}
Démontrons la première assertion, la seconde se démontre de même.
Supposons que $\omega|_{\Int(X)}=0$.
Pour démontrer que $\omega=0$, il suffit de démontrer
que $\omega|_{\Sigma}=0$ pour toute partie squelettique compacte~$\Sigma$
de~$X$.
Soit $\Sigma$ une telle partie; comme $X$ est très riche,
il existe une partie squelettique compacte~$\Tau$ de~$X$,
purement de dimension~$n$, telle que $\Sigma\subset \Tau$.
Par hypothèse, $\omega|_{\Int(X)\cap\Tau}=0$.
Puisque $\partial(X)\cap\Tau$ est $n$-négligeable,
$\Int(X)\cap\Tau$ est dense dans~$\Tau$.
Comme $\omega$ est à coefficients tropicalement continus,
on en déduit $\omega|_\Tau=0$ et, en particulier, $\omega|_\Sigma=0$,
comme il fallait démontrer.
\end{proof}

Rappelons (corollaire~\ref{coro.courant-L1loc})
que toute fonction localement intégrable~$u$,
a fortiori toute fonction continue~$u$, définit un courant~$[u]$
sur~$X$ de bidimension~$(n,n)$, donc de bidegré~$(0,0)$.
\begin{coro}
Soit $X$ un espace tropical de niveau~$n$, très riche. 
L'homomorphisme $u\mapsto [u]$ de $\mathscr C_X$ dans~$\mathscr D^{0,0}_X$
est injectif. 
\end{coro}
\begin{proof}
Soit $U$ un ouvert de~$X$ et $u\in\mathscr C_X(U)$ telle que $[u]=0$.
Démontrons que $u$ est nulle. 
Raisonnons par l'absurde. Soit $x$ un point de~$U$ tel que 
$u(x)\neq 0$ et soit~$V$ un voisinage ouvert de~$x$ dans~$U$ 
sur lequel $u$~ne s'annule pas. 
Il existe une fonction continue~$v$ sur~$U$,
à support compact contenu dans~$V$, 
telle que $uv$ soit lisse et non nulle.
On a alors $[uv]=v[u]=0$, donc $uv=0$, ce qui contredit
le corollaire~\ref{coro.injectivite-p,q}.
\end{proof}

\begin{coro}
Soit $X$ un espace tropical de niveau~$n$, très riche. 
Le support du courant~$\delta_X$ est égal à~$X$.
\end{coro}
\begin{proof}
Soit $U$ un ouvert de~$X$ tel que $\delta_X|_U=0$.
On $\delta_X=[1]$, et $\delta_X|_U=[1_U]$. Comme $U$ est un ouvert
d'un espace très riche, il est très riche. Par suite, $1_U=0$,
ce qui signifie que $U$ est vide.
\end{proof}

\begin{lemm}\label{lemm.limite-mesure-courant}
Soit $U$ un ouvert de~$X$ tel que $d_\trop(x)<n$ 
pour tout $x\in X\setminus U$.
Soit $(u_n)$ une suite 
de fonctions localement localement intégrables sur~$U$,
convergeant simplement vers une fonction~$u$.
S'il existe une fonction localement intégrable~$v$ sur~$X$
telle que $\abs{u_n}\leq v|_U$, alors
Alors la suite de courants $([u_n])$ converge vers le courant~$[u]$ sur~$X$.
\end{lemm}
Dans cet énoncé, les fonctions~$u_n$ (\resp~$u$) définissent a priori des courants~$[u_n]$
de bidimension~$(n,n)$ sur~$U$; d'après le lemme~\ref{lemm.prolongement-courant-trop},  ces courants se prolongement par zéro à~$X$ et on a
noté également~$[u_n]$ et~$[u]$ ces prolongements.
\begin{proof}
Soit $\alpha$ une $(n,n)$-forme lisse à support compact sur~$X$.
Son support~$S$ est une partie compacte de~$U$, par hypothèse
sur la dimension tropicale des points de~$X\setminus U$.
Pour tout~$n$, on a $\langle[u_n],\alpha\rangle=\int_S u_n \alpha_\Rad$,
et de même pour~$u$.
La mesure $\abs v \alpha_\Bor=\abs{ v\alpha_\Rad}$
étant finie, il résulte donc du théorème de convergence dominée que
$\langle[u_n],\alpha\rangle\to \langle[u],\alpha\rangle$.
\end{proof}

\section{Courants positifs et G-topologie}
\label{sec.courants-positifs-G}

Le but de ce chapitre est d'étendre aux G-formes
l'action des courants positifs, 
c'est-à-dire, étant donné un courant positif~$T$ sur~$X$,
définir $\langle T,\alpha\rangle$ pour toute G-forme~$\alpha$
(à support compact)
sur un domaine~$V$ (non nécessairement ouvert) de~$X$,
et de montrer que les courants ainsi étendus forment un G-faisceau.

%
%
 
\begin{lemm}\label{lemm.inegalite-semi-locale}
Soit $\mathscr B$ une classe de fonctions numériques.
Soit $X$ un espace tropical et soit $V$ un domaine fermé de~$X$.
Soit $\omega$ une G-forme de type~$(p,p)$
à coefficients dans~$\mathscr B$ sur~$X$.
On suppose que $\omega|_V$ et $\omega|_{X\setminus V}$ sont positives.
Alors $\omega$ est positive.
\end{lemm}
\begin{proof}
L'assertion est G-locale. On peut donc supposer que $X$ et~$V$
sont compacts et que $\omega$ est tropicale.
Soit $f\colon X\to\R^n$ une carte affine qui tropicalise~$\omega$,
soit $\alpha$ une forme lisse sur~$f(X)$ 
telle que $\omega=f^*\alpha$.
Il suffit de démontrer que $\alpha$ est positive.
En raisonnant G-localement, on peut donc supposer que $f(X)$ 
est une cellule de~$\R^n$.

Posons $P=f(V)$; comme $\omega|_V$ est positive, 
la forme $\alpha|_P$ est positive
(lemme~\ref{lemme-gpositif.compatibilite}).

Posons $U=f(X)\setminus P$ ; alors $f^{-1}(U)\subset X\setminus V$,
de sorte $\omega|_{f^{-1}(U)}$ est positive. 
Cela entraîne que $\alpha|_U$ est positive
(lemme~\ref{lemme-gpositif.compatibilite}).

Par définition, la positivité d'une forme
est une condition ponctuelle sur la cellule~$f(X)$.
Par suite $\alpha$ est positive.
\end{proof}

\begin{prop}\label{prop.mesure-restriction}
Soit $X$ un espace tropical de niveau~$n$, 
localement différentiellement séparé.
Soit $p$ un entier naturel et soit $T$ un courant
positif de bidegré~$(p,p)$
et soit $\omega$ une forme lisse de type~$(p,p)$ sur~$X$.
Soit $V$ un domaine de~$X$ tel que $\omega|_V\geq 0$.
Alors, la restriction à~$V$ de la mesure $\mu^T_\omega$ 
(exemple \ref{coro.courant-positif})
est positive.
\end{prop}
\begin{proof}
L'assertion est locale sur~$X$,
ce qui permet de supposer que $V$ est fermé.
Considérons des formes lisses et positives sur~$X$ 
telles que $\omega=\omega_1-\omega_2$ 
(proposition~\ref{prop.trop-diff-positives}); posons $\alpha=\omega_1+\omega_2$.
Puisque $\pm\omega\leq \alpha$, on observe
l'inégalité de mesures positives
$\abs{\mu^T_\omega}\leq\mu^T_{\alpha}$.
Quitte à remplacer~$X$ par un ouvert relativement compact,
on suppose également que $\mu^T_\alpha$ est finie.

Il s'agit de démontrer l'inégalité $\mu^T_\omega(A)\geq0$ pour
toute partie compacte~$A$ de~$V$.

Fixons $\eps>0$.
Par définition de la mesure d'une partie compacte associée
à une forme linéaire sur l'espace des fonctions continues à support compact
(\cite{INT-1-4}, chap.~4, \S4),
il existe des fonctions continues~$f$ et~$g$ sur~$X$
telles que $\abs{f-\mathbf 1_A}\leq g$ et $\mu^T_\alpha(g)\leq\eps$.
Comme $X$ est localement différentiellement séparé,
il existe, d'après le théorème~\ref{theo.trop-sw}, 
des fonctions lisses~$f'$ et~$g'$ sur~$X$ telles que $\norm{f-f'}\leq\eps$
et $\norm{g-g'}\leq\eps$.
Posons $u=f'+g'+2\eps$ et $v=2g'+4\eps$; ce sont des fonctions lisses
sur~$X$ telles que $0\leq u-\mathbf 1_A\leq v$. De plus,
$ \mu^T_\alpha(v)\leq 2\eps +4\eps \norm{\mu^T_\alpha}$.
Alors,
\[ \abs{\mu^T_\omega(u) - \mu^T_\omega(A)}
\leq \abs{\mu_\omega^T} (u-\mathbf 1_A) 
\leq \abs{\mu^T_\omega}(v) \leq \mu^T_\alpha(v). \]
En particulier, $\mu^T_\omega(A)\geq\mu^T_\omega(u)-\mu^T_\alpha(v)$.

Démontrons maintenant l'inégalité $u\omega \geq -v\alpha$ de
formes lisses sur~$X$. En effet, cette inégalité vaut sur
le domaine fermé~$V$ de~$X$ sur lequel $u\omega$ et $v\alpha$ sont 
des formes positives. D'autre part, sur l'ouvert $X\setminus V$,
on a $\abs u\leq v$, si bien que 
\[ u\omega=u\omega_1-u\omega_2 \geq -v\omega_1-v\omega_2 = -v\alpha. \]
L'assertion découle alors du lemme~\ref{lemm.inegalite-semi-locale}.
Par conséquent, on a $\mu^T_\omega(u)\geq -\mu^T_\alpha(v)$ d'où,
finalement,
\[ \mu^T_\omega(A) \geq -2 \mu^T_\alpha(v) \geq -4 \eps(1+2\norm{\mu^T_\alpha}) .\]
Comme $\eps$ est arbitraire, cela démontre l'inégalité $\mu^T_\omega(A)\geq0$
et conclut la démonstration.
\end{proof}

\begin{coro}\label{coro.independance-extension}
Soit $X$ un espace tropical et soit $T$ un courant positif sur~$X$,
de bidegré~$(p,p)$.
Soit $\omega$ et $\omega'$ des formes de type~$(p,p)$
à support compact sur~$X$
telles que $\omega|_V=\omega'|_V$.
Les restrictions à~$V$ des mesures $\mu^T_\omega$
et $\mu^T_{\omega'}$ coïncident.
\end{coro}
\begin{proof}
La mesure de Radon $\mu^T_{\omega'}-\mu^T_\omega$
est égale à $\mu^T_{\omega'-\omega}$. Puisque la restriction à~$V$
de $\omega'-\omega$ est nulle, la proposition précédente affirme
que la restriction à~$V$ de cette mesure de Radon est à la fois
positive et négative, donc nulle.
L'assertion en résulte.
\end{proof}

\begin{theo}
Soit $X$ un espace G-tropical affiné décent 
et localement différentiellement séparé.
Soit $T$ un courant positif de bidegré~$(p,p)$ sur~$X$.
Il existe un unique morphisme $\R$-linéaire du G-faisceau des G-formes
G-lisses de type~$(p,p)$ dans le G-faisceau des mesures de Radon sur~$X$
qui prolonge le morphisme $\omega\mapsto \mu^T_\omega$ lorsque 
$\omega$ est une forme lisse de type~$(p,p)$ définie sur un ouvert de~$X$.
\end{theo}

Le support de cette mesure $\mu^T_\omega$  est contenu
dans l'intersection des supports de~$T$ et de~$\omega$.

En particulier, pour toute G-forme $\omega$  de type~$(p,p)$
sur un domaine~$V$
de~$X$ tel que $\supp(T)\cap \supp(\omega)$ soit compact dans~$V$,
on dispose d'une mesure à support compact~$\mu^T_\omega$ sur~$V$.
On notera $\langle T,\omega\rangle$ la masse totale de cette mesure,
c'est-à-dire l'intégrale de la fonction constante~$1$ sur~$V$. 
Cette notation est compatible avec celle déjà introduite
lorsque $V$ est ouvert et $\omega$ est une forme lisse.

\begin{proof}
Soit $\omega$ une G-forme G-lisse de type~$(p,p)$ sur un domaine~$V$ de~$X$.
Soit $(V_i)$ un G-recouvrement de~$V$ par des domaines compacts sur
lesquels $\omega$ est une forme. Pour tout~$i$, 
il existe une forme lisse~$\omega_i$ à support  compact sur~$X$
qui prolonge~$\omega|_{V_i}$ (proposition~\ref{prop.decent-extension}, c)).
Nécessairement, la restriction à~$V_i$ de~$\mu^T_\omega$ 
est la restriction à~$V_i$ de~$\mu^T_{\omega_i}$
qui est imposée par les conditions du théorème.
Il s'ensuit que la mesure~$\mu^T_\omega$ est entièrement déterminée,
d'où l'unicité d'un tel morphisme de G-faisceaux.
Établissons maintenant son existence.

Soit $\omega$ une G-forme G-lisse sur un domaine~$V$ de~$X$.
Soit $\mathscr W_V$ la famille des domaines~$W$ de~$V$
tels qu'il existe un voisinage ouvert~$W'$ de~$W$ dans~$X$
et une forme $\omega'$ sur~$W'$ telle que $\omega'|_{W}=\omega|_W$.
Tout domaine compact de~$V$ sur lequel $\omega$ 
est une forme appartient à~$\mathscr W_V$
(proposition~\ref{prop.decent-extension}, c)).
Par conséquent, $\mathscr W_V$ est un G-recouvrement de~$V$.

Soit $W\in\mathscr W_V$ et soit $W'$ et $\omega'$ comme ci-dessus.
D'après le corollaire~\ref{coro.independance-extension},
la mesure $\mu^T_{\omega'}|_{W}$
ne dépend pas du choix de la forme~$\omega'$ ; on la notera $\mu_W$.
Soit $W_1$ et $W_2$ dans~$\mathscr W_V$, soit $W'_1, \omega'_1$ et $W'_2,\omega'_2$ comme ci-dessus;
les restrictions à~$W_1\cap W_2$ de $\omega'_1$ et $\omega'_2$
sont égales à la restriction de~$\omega$; 
d'après le corollaire~\ref{coro.independance-extension},
les mesures $\mu_{W_1}$ et $\mu_{W_2}$ coïncident sur $W_1\cap W_2$.
Il existe donc une unique mesure~$\mu^T_\omega$ sur~$V$ 
dont la restriction à~$W$  est égale à~$\mu_W$ pour tout
domaine~$W$ appartenant à~$\mathscr W_V$.

Lorsque $V$ est ouvert et $\omega$ est une forme,
le domaine~$V$ appartient à~$\mathscr W_V$, de sorte que $\mu^T_\omega$
coïncide avec la mesure~$\mu^T_\omega$ précédemment définie.

Soit $V$ et $V'$ des domaines de~$X$ tels que $V'\subset V$.
Soit $W\in\mathscr W_V$, soit $W'$ et $\omega'$ comme ci-dessus;
la mesure $\mu^T_\omega|_W$ est la trace de la mesure $\mu^T_{\omega'}$.
Le domaine~$W'$ est un voisinage ouvert de~$V'\cap W$
et la forme~$\omega'$ prolonge $\omega|_{V'\cap W}$;
par suite, la restriction à $V'\cap W$
de la mesure $\mu^T_{\omega|_{V'}}$ 
est donnée la trace de la mesure $\mu^T_{\omega'}$;
c'est donc la trace de la mesure $\mu^T_\omega|_W$.
Comme les parties~$W\cap V'$, où $W$ parcourt~$\mathscr W_V$, constituent
un G-recouvrement de~$V'$, la mesure $\mu^T_{\omega|_V'}$ est la restriction
à~$V'$ de la mesure~$\mu^T_\omega$.

Cela prouve que les applications $\omega\mapsto \mu^T_\omega$ 
définissent un morphisme de G-faisceaux. 

Il reste à démontrer qu'il est $\R$-linéaire.
Soit $\alpha$ et $\beta$ des G-formes G-lisses de type~$(p,p)$
sur un domaine~$V$ de~$X$,
soit $a$ et $b$ des nombres réels.
Les domaines~$W$ de~$V$ qui possèdent un voisinage ouvert~$W'$
sur lequel les G-formes $\alpha|_W$ et $\beta|_W$ se prolongent
en une forme définissent un G-recouvrement de~$V$.
Pour tout tel domaine~$W$, la restriction à~$W$
de la mesure $\mu^T_{a\alpha+b\beta}$ est égale 
à la trace de la mesure $a\mu^T_\alpha+b\mu^T_\beta$.
Par suite, $\mu^T_{a\alpha+b\beta}= a\mu^T_\alpha+b\mu^T_\beta$.
\end{proof}

\begin{prop}\label{prop.mu-T-omega-pos}
Soit $X$ un espace G-tropical affiné décent 
et localement différentiellement séparé.
Soit $T$ un courant positif sur~$X$, de bidegré~$(p,p)$.

Si $\omega$ est une G-forme G-lisse positive sur un domaine~$V$ de~$X$,
la mesure $\mu^T_\omega$ sur~$V$ est positive.
\end{prop}
\begin{proof}
Soit $\omega$ une G-forme G-lisse positive sur~$X$.
La positivité de la mesure $\mu^T_\omega$ se vérifie G-localement.
Il suffit donc de prouver que pour domaine compact~$V$ tel que $\omega|_V$
se prolonge en une forme~$\omega'$ sur un voisinage ouvert~$V'$ de~$V$,
la restriction à~$V$ de la mesure~$\mu^T_\omega$ est positive. 
Or, par construction, la restriction à~$V$ de cette mesure~$\mu^T_\omega$
est égale à la restriction à~$V$ de~$\mu^T_{\omega'}$,
laquelle est positive d'après la proposition~\ref{prop.mesure-restriction}
puisque $\omega'|_V=\omega|_V$ est positive.
\end{proof}
\begin{coro}
L'application $\omega\mapsto \langle T,\omega\rangle$ 
est continue lorsque l'espace $\mathscr A^{p,p}_{\groth, \cpct}(X)$ 
est muni de sa topologie~$\mathscr C^0$.
\end{coro}
\begin{proof}
Soit $\mathscr D$ un jeu de cartes G-tropicales G-couvrant sur~$X$ et soit $F$
une partie compacte de~$X$. Il s'agit d'établir
la continuité de la restriction de cette forme linéaire 
à l'espace $\mathscr A^{p,p}_{\groth, \mathscr D, F}$ 
des G-formes sur~$X$ qui sont G-tropicalisées par le jeu~$\mathscr D$
et dont le support est contenu dans~$F$.
Comme les inclusions induites par les raffinements de jeux de cartes G-couvrants
sont continues, il sera loisible de raffiner~$\mathscr D$ au cours
de la démonstration.

Toute carte G-tropicale de~$\mathscr D$ est G-recouverte par des cartes
tropicales de domaines compacts. Quitte à raffiner le jeu~$\mathscr D$,
on peut donc supposer qu'il est constitué de  cartes tropicales
définies sur des domaines compacts.

Soit $f\colon W\to E$ une telle carte. Comme $X$ est décent,
il existe un recouvrement fini $(W_i)$ de~$W$ par des domaines compacts
et pour tout~$i$, une fonction affine $f_i$ sur un voisinage ouvert~$U_i$
de~$W_i$ telle que $f_i|_{W_i}=f|_{W_i}$.

Quitte à raffiner le jeu~$\mathscr D$, on peut donc supposer
que pour toute carte $f\colon W\to E$, 
il existe un voisinage domanial compact~$W'$
de~$W$ sur lequel $f$ se prolonge en une application affine que nous
noterons encore~$f$.
La famille $\mathscr E$ des cartes tropicales $(W',f)$ est alors
un jeu de cartes tropicales sur~$X$ qui est couvrant.
On munit les espaces $\mathscr A^{p,p}_{f,\groth}(W)$ 
et $\mathscr A^{p,p}_{f,\groth}(W')$  de normes standard~$\mathscr C^0$.

Soit $S$ une partie finie de~$\mathscr D$ telle que la réunion des~$W$,
pour $W\in S$, contienne~$F$. Pour toute partie non vide~$P$ de~$S$,
notons~$W_P$ l'intersection des~$W$, pour $W\in P$. Par la formule 
d'inclusion-exclusion, on a 
\[ \mathbf 1_W 
   = \sum_{\emptyset\neq P \subset S} (-1)^{\Card(P)-1} \mathbf 1_{W_P}. \]
Soit $\omega\in\mathscr A^{p,p}_{\groth, \mathscr D,F}(X)$.
Comme le support de~$\omega$ est contenu dans~$F$, on a donc 
\[ \langle T,\omega\rangle = \mu^T_\omega(1) = \mu^T_\omega(\mathbf 1_W)
 = \sum_{\emptyset\neq P} (-1)^{\Card(P)-1} \mu^T_\omega(\mathbf 1_{W_P}) . \]
Il suffit de prouver la continuité de chacun des termes de cette somme.
Chacun d'entre eux est du type $\omega\mapsto \mu^T_\omega(\mathbf 1_C)$,
où $C$ est un domaine compact d'une carte~$W$ de~$\mathscr D$.
Il suffit donc de prouver que pour tel domaine~$C$ d'une carte~$W$,
l'application $\omega\mapsto\mu^T_\omega(\mathbf 1_C)$ 
sur $\mathscr A^{p,p}_{\groth, \mathscr D,F}(X)$ est continue.

Notons $E$ l'espace affine codomaine de la carte~$f$ et munissons-le
d'une norme euclidienne ; notons $\eta=\ddc \norm{x}^2$
la $(1,1)$-forme fortement positive sur~$E$.
Soit $P$ l'espace paralinéaire compact~$f(W)$; 
considérons une norme~$\mathscr C^0$ sur l'espace~$\mathscr A^{p,p}(P)$. 
On déduit
de la proposition~\ref{prop.harvey} qu'il existe un nombre réel $c>0$
tel que la forme $\alpha+ c \norm\alpha \eta^p$ est 
positive, pour tout $\alpha\in\mathscr A^{p,p}(P)$.

Pour toute G-forme $\omega\in\mathscr A^{p,p}_{\groth,\mathscr D,F}(X)$,
il existe une unique forme $\alpha\in\mathscr A^{p,p}(P)$
telle que $\omega|_W=f^*\alpha$ et l'application $\omega\mapsto \norm\omega_{P,0}$ est une semi-norme~$\mathscr C^0$ sur~$\mathscr A^{p,p}_{\groth,\mathscr D,F}(X)$. On la note $\norm{\cdot}_{W,0}$.

Soit $\omega\in\mathscr A^{p,p}_{\groth,\mathscr D,F}(X)$;
il existe une unique forme $\alpha\in\mathscr A^{p,p}(P)$
telle que $\omega|_W=f^*\alpha$, 
si bien que $\omega + c\norm \alpha f^*\eta^p$ est positive sur~$W$.
Par suite, la mesure $\mu^T_\omega + c \norm\alpha \mu^T_{f^*\eta^p}$
sur~$W$ est positive. En particulier, 
\[ \mu^T_\omega(\mathbf 1_C) \geq - c \mu^T_{f^*\eta^p}(\mathbf 1_C)
  \norm{\omega}_{W,0}. \]
On en déduit que 
\[ \abs{\mu^T_\omega(\mathbf 1_C)}\leq
 c \mu^T_{f^*\eta^p}(\mathbf 1_C) \norm{\omega}_{W,0}, \]
ce qui conclut la démonstration.
\end{proof}

\subsection{}
Si $\alpha$ est une G-forme sur~$X$ et $T$ un courant positif sur~$X$,
on définit ainsi un courant $T\wedge \alpha$ sur~$X$ :
il applique une forme $\omega$ sur $\langle T, \alpha\wedge\omega\rangle$.
Sa continuité résulte du corollaire précédent; c'est même un courant
d'ordre~$0$.

\begin{prop}\label{prop.pos-wedge-G-pos}
Soit $T$ un courant positif sur~$X$ et soit $\alpha$ une G-forme sur~$X$.
Si $\alpha$ est fortement positive, alors $T\wedge\alpha$ est un courant
positif.
\end{prop}
\begin{proof}
Soit $\omega$ une forme lisse à support compact sur~$X$.
Par définition, le courant $T\wedge\alpha$ applique~$\omega$
sur $\mu_{\alpha\wedge\omega}^T(1)$. 
Il suffit donc de démontrer que 
la mesure $\mu_{\alpha\wedge\omega}^T$ est positive,
et, d'après la proposition~\ref{prop.mu-T-omega-pos}, 
il suffit de prouver que la G-forme $\alpha\wedge\omega$ est positive.
Comme $\omega$
est positive et $\alpha$ est fortement positive,
cela résulte de la proposition~\ref{prop.produit-formes-positives-berk}.
\end{proof} 

\subsection{}
La construction s'étend par linéarité (et raisonnement local) 
aux courants qui sont localement
différence de courants positifs.

\section{Courants squelettiques}

\begin{defi}
Soit $X$ un espace tropical de niveau~$n$, localement
affinement séparé, soit $p$ un entier naturel 
et soit~$T$ un courant de bidimension~$(p,p)$ sur~$X$.

\begin{enumerate}
\item
Soit $f$ un carte tropicale sur~$X$ définie sur un ouvert~$U$ de~$X$
et soit $P$ une partie squelettique  fermée de~$U$ de dimension~$\leq p$,
munie d'un $p$-calibrage \emph{G-localement constant}.
On dit que $P$ décrit le courant~$T$
relativement à la carte~$f$ si, pour tout $(p,p)$-forme  lisse~$\omega$
sur~$U$ qui est tropicalisée par~$f$ et toute fonction lisse~$\rho$
à support compact sur~$U$, on a 
\[ \langle T,\rho\omega \rangle = \int_{P}  \rho\omega|_{P}. \]

\item
On dit que $T$ est un \emph{courant squelettique} si, pour tout ouvert~$U$
de~$X$ et toute carte tropicale~$f$ sur~$U$, le courant $T|_U$ est 
localement descriptible relativement au moment~$f$ par une partie
squelettique munie d'un $p$-calibrage G-localement constant.\index{courant squelettique}
\end{enumerate}
\end{defi}

Rappelons (remarque~\ref{rema.formes-test})
qu'un courant est déterminé par sa valeur sur les
formes du type $\rho\omega$ considérées dans la définition,
dès lors que l'espace ambiant est localement différentiellement séparé.

Un courant squelettique est symétrique et d'ordre~$0$.

Si les calibrages des espaces squelettiques
calibrés décrivant un courant squelettique sont positifs, 
ce courant est fortement positif.

Les courants squelettiques forment un sous-faisceau
en $\R$-espaces vectoriels du faisceau des courants de bidimension~$(p,p)$
sur~$X$.

Par définition, les courants squelettiques ont une description
en termes de squelettes purement de dimension~$p$;
nous parlerons ainsi de « courant squelettique de dimension~$p$ ».

\begin{exem}
Soit $X$ un espace tropical de niveau~$n$, localement différentiellement
séparé,
et soit $P$ un sous-espace squelettique de~$X$, 
fermé et de dimension~$p$, muni d'un $p$-calibrage G-localement constant.
L'application~$\delta_P$ qui applique une forme lisse à support
compact~$\omega$ sur~$X$ sur l'intégrale $\int_P \omega|_P$
est un courant squelettique.
\end{exem}

\begin{exem}\label{exem.deltaX-paralineaire}
Soit $X$ un espace tropical de niveau~$n$, sans bord.
Le courant d'intégration sur~$X$ est un courant
squelettique de dimension~$n$:
pour toute carte tropicale~$f$ 
sur un ouvert de~$X$,
il est décrit par le polyèdre de variation maximale~$\Sigma_f^{(n)}$
muni de son $n$-calibrage canonique.
\end{exem}

\begin{lemm}\label{lemm.courant-paralineaire-restriction}
Supposons que $P_f$ décrive~$T$ relativement à une carte
tropicale~$f\colon V\to E$ définie sur un ouvert~$V$ de~$X$.
Alors, pour tout ouvert~$W$ de~$V$, le sous-espace calibré
$P_f\cap W$ décrit~$T$ relativement  à la carte~$f|_W$ sur~$W$.
\end{lemm}
\begin{proof}
Soit $\omega$ une $(p,p)$-forme lisse sur~$W$ qui est
tropicalisée par~$f|_W$ et soit $\rho$ une fonction lisse à support
compact sur~$W$.
Soit $\alpha$ une $(p,p)$-forme lisse au voisinage de~$f(W)$
telle que $\omega=(f|_W)^*\alpha$ (cf.~\S\ref{sss.Gpq-pf}).

Soit $K$ le support de~$\rho$; alors $f(K)$ est une partie
compacte de~$f(W)$ contenue dans l'ouvert de définition de~$\alpha$.
En multipliant~$\alpha$ par une fonction lisse sur~$E$ qui vaut~$1$ 
au voisinage de~$f(K)$ et dont  le support est contenu
dans l'ouvert de définition de~$\alpha$, on obtient une forme~$\alpha'$
sur~$E$ 
qui coïncide avec~$\alpha$ au voisinage de~$f(K)$.
Soit~$\rho'$ le prolongement par~$0$ de~$\rho$ à~$V$.
On a $\rho' f^*\alpha'|_W = \rho (f|_W)^* \alpha $
et $\omega'=\rho' f^*\alpha'$ est le prolongement par zéro
de $\omega=\rho (f|_W)^*\alpha$.

On a donc
\[ \langle T|_W, \rho\omega \rangle
 = \langle T, \rho'\omega' \rangle
= \int_{P_f} \rho'\omega'|_{P_f} 
 = \int_{P_f\cap W} \rho'\omega'|_{P_f\cap W}
 = \int_{P_f\cap W} \rho\omega|_{P_f\cap W}. \]
Cela prouve que le squelette $p$-calibré
$P_f\cap W$ décrit~$T$ relativement au moment~$f|_W$.
\end{proof}

\subsection{}\label{ss.support}
Soit $P$ une partie squelettique de dimension~$\leq p$ munie
d'un $p$-calibrage G-localement constant~$\mu_P$.
Le support de~$\mu_P$ est une partie squelettique fermée de~$P$,
purement de dimension~$p$ ; on le notera $\abs{(P,\mu_P)}$,
ou $\abs{P}$ si le calibrage~$\mu_P$ est sous-entendu.
Il est naturellement muni du calibrage restreint, qui est partout non nul.

Dans une description d'un courant squelettique, on peut ainsi
toujours supposer que la partie squelettique  considérée
est purement de dimension~$p$ et que son calibrage  est partout non nul.

\subsection{}
Soit $P$ un espace paralinéaire purement de dimension~$p$ et
soit $f\colon P\to\R^n$ une application paralinéaire;
notons~$P^{(f)}$ l'ensemble des points 
$x\in P$ tels que $\dim(f(Q))=p$
pour tout voisinage paralinéaire~$Q$ de~$x$ dans~$P$.
Démontrons que $P^{(f)}$ est une partie paralinéaire fermée de~$P$, 
purement de dimension~$p$, et que la restriction de~$f$ à~$P^{(f)}$
est une immersion par morceaux.
Il suffit de le vérifier localement, ce qui permet de supposer que
$P$ possède une décomposition cellulaire~$\mathscr C$
adaptée à~$f$. Dans ce cas, l'ensemble~$P^{(f)}$ 
est la réunion des $p$-cellules~$C$
de~$\mathscr C$ telle que $f|_C$ est injective, d'où l'assertion.

On vérifie à l'aide d'une décomposition cellulaire
que pour toute  partie paralinéaire compacte~$Q$ de~$P\setminus P^{(f)}$,
on a $\dim(f(Q))<p$.


Dans le cas où $P$ décrit un courant squelettique relativement à
une carte tropicale~$f$,
alors $P^{(f)}$ le décrit également ;  cela permettra de supposer
que le calibrage de~$P$ est partout non nul et 
que $f|_P$ est une immersion par morceaux.

\begin{prop}\label{prop.comparaison-polyedres-T}
Soit $X$ un espace tropical de niveau~$n$, localement
affinement séparé,
soit $p$ un entier naturel
et soit $T$ un courant squelettique de dimension~$p$ sur~$X$.
Soit $f\colon X\to E$ et $g\colon X\to E'$ des cartes tropicales sur~$X$,
soit $P_f$ et~$P_g$ des parties squelettiques $p$-calibrés 
fermées décrivant le courant~$T$
relativement aux cartes~$f$ et~$g$.
On a l'égalité de parties squelettiques $p$-calibrées :
$(\abs{P_f}^{(f)})^{(g)}=(\abs{P_g}^{(g)})^{(f)}$.
\end{prop}
%

\begin{proof}
On peut supposer que $P_f=\abs{P_f}^{(f)}$ et que $P_g=\abs{P_g}^{(g)}$;
il s'agit alors de démontrer l'égalité 
de polyèdres calibrés $P_f^{(g)}=P_g^{(f)}$.

Introduisons la carte tropicale $h=(f,g)$ sur~$X$, 
soit $P_h$ une partie squelettique $p$-calibrée qui décrit le courant~$T$
relativement à la carte~$h$ et telle que $P_h=\abs{P_h}^{(h)}$.
Il suffit de prouver que l'on a l'égalité  $P_f=P_h^{(f)}$ de parties
squelettiques $p$-calibrées.
En effet, comme $f$ est une immersion par morceaux
sur~$P_f$, c'est aussi le cas de~$h$,
d'où $P_f^{(h)}=P_f$ et
$P_f^{(g)}=(P_f^{(h)})^{(g)}=P_h^{(f)(g)}$;
par symétrie, on a aussi $P_g^{(f)}=P_h^{(g)(f)}$.
Ainsi, l'égalité $P_f^{(g)}=P_g^{(f)}$ de parties squelettiques
calibrées équivaut à l'égalité
$P_h^{(f)(g)}=P_h^{(g)(f)}$.
Cette dernière égalité résulte de ce que 
les deux espaces paralinéaires $P_h^{(f)(g)}$ et $P_h^{(g)(f)}$
coïncident: si $\mathscr C$ est une décomposition cellulaire
assez fine de~$P_h$, c'est la réunion des cellules sur lesquelles
$f$ et~$g$ sont simultanément injectives.

Prouvons maintenant l'égalité $P_f=P_h^{(f)}$.
Posons $P=P_f\cup P_h^{(f)}$.
C'est une partie squelettique de~$X$,
purement de dimension~$p$, munie de deux calibrages,
$\mu$ et~$\mu'$, respectivement déduits par prolongement par~$0$
de celui de~$P_f$ et de celui de~$P_h^{(f)}$.
Il suffit de démontrer que $\mu=\mu'$.
Notons que $f$ et $h$ sont des immersions par morceaux sur~$P$.

L'espace $P$ est G-recouvert par des $h$-cellules compactes~$C$
de dimension~$p$ telles que $\mathring C$ soit ouvert dans~$P$
(ce sont également des $f$-cellules);
on peut également supposer que l'intérieur de ces cellules
est soit entièrement contenu dans~$P_f$, 
soit entièrement contenu dans~$P_h^{(f)}$,
et que les calibrage~$\mu|_C$ et~$\mu'|_C$ sont constants relativement à~$f$.

Soit $C$ une telle cellule et soit $V$ un ouvert de~$X$ 
tel que $V\cap C=\mathring C$.

Démontrons que $(\mathring C,\mu|_{\mathring C})$ décrit le courant~$T|_V$
relativement à~$f$.
C'est effectivement le cas si $C$ est une cellule de~$P_f$,
d'après le lemme~\ref{lemm.courant-paralineaire-restriction}.
C'est également le cas si $\mathring C$ est disjointe de~$P_f$, 
car alors le courant~$T|_V$ est nul en toute forme $f$-tropicale
(d'après le lemme~\ref{lemm.courant-paralineaire-restriction}),
de même que le calibrage $\mu_{\mathring C}$, par définition de~$\mu$.

Démontrons que $(\mathring C,\mu'|_{\mathring C})$ décrit
le courant~$T|_V$ relativement à~$h$.
Si $C$ est une cellule de~$P_h$,  
cela résulte du lemme~\ref{lemm.courant-paralineaire-restriction}.
Sinon, $\mathring C$ est disjointe de~$P_h$,
donc le courant~$T|_V$ est nul en toute forme $h$-tropicale,
et $\mu'|_{\mathring C}=0$.

Soit $\alpha$ une $(p,p)$-forme sur~$\R^m$, 
invariante par translation et non nulle sur $f(C)$.
La forme $\omega=f^*\alpha$ est $f$-tropicale par construction,
elle également $h$-tropicale car $h=(f,g)$.
Soit $\rho$ une fonction lisse à support compact sur~$V$, 
positive, et dont la restriction à~$\mathring C$ n'est pas nulle.
On a 
\[ \langle T,\rho\omega \rangle
= \int_{P_f} \rho f^*\alpha|_{P_f} 
= \int_{f(C)} \rho_C  \langle \alpha,\mu_C\rangle, \]
où $\rho_C$ est la fonction sur~$f(C)$
déduite de~$\rho$ par restriction à~$C$
et l'isomorphisme~$f|_C$, 
et $\mu_C$ est le vecteur-volume sur~$f(C)$
tel que $\mu|_C=f^*\mu_C$.

Notons de même $\rho'_C$ la fonction sur~$h(C)$
déduite de~$\rho$ par~ l'isomorphisme~$h|_C$.
La projection~$q$ de~$\R^{m+n}$ sur~$\R^m$
induit un isomorphisme de $h(C)$
sur $f(C)$. 
Notons $\mu'_C$ le vecteur-volume sur $h(C)$ tel que
$\mu'|_C=h^*\mu'_C$; on a alors
\[ \langle T,\rho\omega \rangle
= \int_{P_h} \rho h^*(q^*\alpha)|_{P_h} 
= \int_{h(C)} \rho'_C  \langle q^*\alpha,\mu'_C\rangle
 = \int_{f(C)} \rho_C \langle \alpha, q_{*}\mu'_C\rangle. \]
On a donc
\[  \int_{f(C)} \rho_C  \langle \alpha,\mu_C\rangle
= \int_{f(C)} \rho_C \langle \alpha,q_*(\mu'_C)\rangle.\]
Comme la forme~$\alpha$ est arbitraire,  
les calibrages $\mu_C$ et~$q_*(\mu'_C)$ sont constants,
et la fonction~$\rho_C$ est positive et non nulle, 
on en déduit  l'égalité $\mu_C=q_*(\mu'_C)$.
On a donc $\mu'|_C=\mu|_C$.
Cela termine la démonstration.
\end{proof}


Par exemple, lorsque $T$ est le courant~$\delta_P$
associé à une partie squelettique $p$-calibrée~$P$ de~$X$, 
on a $P_f^T=(\abs P\cap V)^{(f)}$
pour tout ouvert~$V$ et toute carte tropicale~$f$ sur~$V$.

Si $f$ et~$g$
sont des cartes tropicales sur un ouvert~$V$ de~$X$,
la conclusion de la proposition~\ref{prop.comparaison-polyedres-T}
se récrit $(P_f^T)^{(g)}=(P_g^T)^{(f)}$.
 
\begin{prop}\label{prop.courant-paralineaire-faisceau}
Soit $X$ un espace tropical de niveau~$n$, localement
affinement séparé, soit $p$ un entier naturel
et soit $T$ un courant squelettique de dimension~$p$ sur~$X$.
Soit $f$ une carte tropicale sur~$X$.

\begin{enumerate}
\item Il existe une partie squelettique 
fermée~$P$ de~$X$, munie d'un $p$-calibrage G-localement constant, qui décrit~$T$
relativement à~$f$.
\item Soit $P$ une telle partie squelettique.
La partie squelettique $p$-calibrée~$\abs P^{(f)}$
est la plus petite squelettique fermée $p$-calibrée
de $X$ décrivant~$T$ relativement à~$f$, 
et c'est la seule partie squelettique fermée $p$-calibrée
décrivant~$T$ relativement à~$f$ dont le calibrage soit partout non nul
et sur laquelle $f$ soit une immersion par morceaux.
\end{enumerate}
\end{prop}
\begin{proof}
Par hypothèse, il existe un recouvrement ouvert~$(U_i)$ de~$X$
et, pour tout~$i$, une partie squelettique~$P_i$ de~$U_i$,
munie d'un $p$-calibrage G-localement constant, qui décrit le courant~$T|_{U_i}$
relativement à~$f|_{U_i}$.
On peut supposer que chaque $P_i$ est purement de dimension~$p$
et que le calibrage y est partout non nul
et que $f|_{P_i}$ est une immersion par morceaux,
autrement dit, que $P_i=\abs{P_i}^{(f)}$.

Comme $P_i\cap (U_i\cap U_j)$ et $P_j\cap (U_i\cap U_j)$ décrivent
toutes deux le courant $T|_{U_i\cap U_j}$ relativement
à $f|_{U_i\cap U_j}$, on a $ P_i\cap (U_i\cap U_j)=P_j\cap (U_i\cap U_j)$
(proposition~\ref{prop.comparaison-polyedres-T}).
Il existe donc une unique paralinéaire squelettique fermée~$P$ de~$X$
tel que $P\cap U_i=P_i$ pour tout~$i$.
De plus, il existe un unique $p$-calibrage sur~$P$
dont la trace sur~$U_i$ coïncide avec le $p$-calibrage de~$P_i$;
il est G-localement constant.

Démontrons que $P$, muni de ce calibrage, décrit~$T$ relativement à la
carte tropicale~$f$.

Soit $\omega$ une forme $f$-tropicale sur~$X$ et soit $\rho$
une fonction lisse à support compact sur~$X$.
Soit $(\theta_i)$ une famille localement finie de fonctions lisses
à support compact, positives, telle que pour tout~$i$, le support de~$\theta_i$
soit contenu dans~$U_i$, et que $\theta=\sum\theta_i$ vaille~$1$
au voisinage du support de~$\rho$
(corollaire~\ref{prop.trop-partition.1.lisse-compact}).

 Alors, $\omega=\rho\omega=\rho\theta\omega
=\sum \rho\theta_i\omega$. Par suite,
\[
 \langle T,\rho\omega\rangle  =\sum \langle T,\rho\theta_i\omega\rangle 
 = \sum \int_{P_i} \rho\theta_i \omega 
 = \sum \int_P \rho\theta_i \omega 
 = \int_P \rho\omega. 
\]
Cela conclut la démonstration de (1). 

Montrons maintenant~(2). Soit~$P$ comme dans (1).
Il est clair que $\abs P^{(f)}$
décrit encore~$T$ relativement à~$f$.
Soit~$Q$ une partie squelettique fermée de~$X$
munie d'un calibrage G-localement constant et décrivant~$T$ relativement à~$f$. 
Appliquons
la proposition~\ref{prop.comparaison-polyedres-T} lorsque $g=f$.
On en déduit que $\abs Q^{(f)}=\abs P^{(f)}$ ; en particulier, 
$\abs P^{(f)}\subset Q$, d'où
la minimalité de $\abs P^{(f)}$.
Par ailleurs, le calibrage de $\abs P^{(f)}$ est partout non nul,
et $f|_{\abs P^{(f)}}$ est une immersion par morceaux. Soit 
maintenant~$Q$ une partie squelettique fermée de~$X$
munie d'un calibrage G-localement constant partout non nul 
et décrivant~$T$ relativement à~$f$,
et telle que $f|_Q$
soit une immersion par morceaux.
On a alors $\abs Q^{(f)}=Q$ ;
en appliquant ce qui précède à $Q$ au lieu de $P$, 
on voit que $\abs Q^{(f)}$ est la plus petite partie
polyédrale fermée calibrée de $X$ décrivant $T$ relativement 
à $f$ ; il vient $\abs Q^{(f)}=\abs P^{(f)}$, c'est-à-dire
$Q=\abs P^{(f)}$. 
\end{proof}

\subsection{}\label{ss.PfT}
Nous noterons~$P_f^T$ la partie squelettique $p$-calibrée minimale 
décrivant le courant $T$ relativement à $f$. Elle est pure
et $f|_{P_f^T}$ est une immersion par morceaux.

Par exemple, lorsque $T$ est le courant~$\delta_P$
associé à une partie squelettique $p$-calibré~$P$ de~$X$, 
on a $P_f^T=(\abs P\cap V)^{(f)}$
pour tout ouvert~$V$ et toute carte tropicale~$f$ sur~$V$.

Soit $f$ une carte tropicale sur un ouvert~$V$ de~$X$
et soit $U$ une partie ouverte de~$X$.
Alors $P_f^T\cap U$ décrit~$T$ relativement à~$f|_U$
(lemme~{lemm.courant-paralineaire-restriction})
et vérifie les conditions de l'assertion~\emph b)
de la proposition~\ref{prop.courant-paralineaire-faisceau};
c'est donc $P_{f|_U}^T$.

En particulier, $P_f^T$ est contenu dans le support de~$T$,
car le polyèdre vide représente le courant nul.

Si $f$ et~$g$
sont des cartes tropicales sur un ouvert~$V$ de~$X$,
la conclusion de la proposition~\ref{prop.comparaison-polyedres-T}
se récrit $(P_f^T)^{(g)}=(P_g^T)^{(f)}$.

\begin{lemm}\label{lemm.squelettique-positif}
Soit $X$ un espace tropical de niveau~$n$, localement 
différentiellement séparé, et soit $T$ un courant squelettique
de bidimension~$(p,p)$ sur~$X$.
Les conditions suivantes sont équivalentes:
\begin{enumerate}\def\theenumi{\roman{enumi}}\def\labelenumi{(\theenumi)}
\item Le courant~$T$ est faiblement positif;
\item Le courant~$T$ est positif;
\item Le courant~$T$ est fortement positif;
\item Pour toute carte tropicale sur~$X$ définie sur un ouvert~$U$,
le calibrage de~$P_f^T$ est positif.
\end{enumerate}
\end{lemm}
En particulier,
les trois notions de positivité des courants coïncident
sur les courants squelettiques.

\begin{proof}
Les implications (iii)$\Rightarrow$(ii) et (ii)$\Rightarrow$(i) sont évidentes.

Démontrons l'implication (i)$\Rightarrow$(iv).
Supposons que $T$ est faiblement positif.
Soit $f$ une carte tropicale définie sur un ouvert~$U$ de~$X$.
Pour démontrer que le calibrage de~$P_f^T$ est positif,
il suffit de démontrer que sa restriction à toute $f$-cellule~$C$,
compacte et de dimension~$p$, de~$P_f^T$ est positif. 
On peut également supposer que le calibrage de~$C$ est standard,
de la forme $f^*\mu$, où $\mu$ est un $p$-vecteur volume 
sur~$\langle f(C)\rangle$. Par définition de~$P_f^T$, $\mu$ n'est pas nul.
Soit $x\in\mathring C$;
observons que $\mathring C$ est ouvert dans~$P_f^T$.
Soit $\rho$ une fonction lisse, positive, sur~$X$ 
égale à~$1$ au voisinage de~$x$ et 
dont le support ne rencontre~$P_f^T$ qu'en~$\mathring C$.
Soit $\alpha$ une $(p,p)$-forme positive standard, non nulle, sur $f(C)$,
et soit $\omega$ le prolongement par~$0$ de $\rho f^*\alpha$.
Par définition de~$P_f^T$, on a 
\[ \langle T,\rho\omega\rangle = \int_{P_f^T} \rho\omega|_{P_f^T}
 = \int_{C} \rho \langle \alpha, \mu\rangle. \]
Puisque cette expression est positive et non nulle,
on a $\langle\alpha, \mu\rangle>0$. Cela prouve que le vecteur-volume~$\mu$
est positif.

Supposons~(iv) et démontrons que le courant~$T$ est fortement positif.
Soit $\omega\in\mathscr A^{(p,p)}_\cpct(X)$ une forme lisse à support
compact sur~$X$. Soit $(U_i)_{i\in I}$ un recouvrement ouvert fini
d'un voisinage du support de~$\omega$ sur les membres desquels $\omega$
est tropicale. Pour tout~$i$, soit $f_i$ une carte tropicale
sur~$U_i$ qui tropicalise~$\omega|_{U_i}$.
Soit $(\rho_i)$ une partition de l'unité lisse au voisinage de~$\supp(\omega)$
subordonnée au recouvrement~$(U_i)$. 
Puisque $\sum\rho_i$ est égale à~$1$ au voisinage de~$\supp(\omega)$,
on a $\omega=\sum_i \rho_i\omega$. Alors,
\[ \langle T,\omega\rangle = \sum_i \langle T,\rho_i\omega\rangle
 = \sum_i \int_{P_{f_i}^T} \rho_i \omega|_{P_{f_i}^T}. \]
Si $\omega$ est faiblement positive, alors 
$\langle \omega|_{P_{f_i}^T}, \mu_{P_{f_i}^T}\rangle$ 
est une mesure positive sur~$P_{f_i}^T$. 
Il en résulte que $\langle T,\omega\rangle$ est positif, ce qu'il fallait démontrer.
\end{proof}

\begin{prop}\label{rema.courant-para-polyedre}
Soit $X$ un espace tropical de niveau~$n$,
localement différentiellement séparé, et soit $p$ un entier naturel.
Pour toute carte tropicale~$f$ sur un ouvert~$V$ de~$X$, donnons-nous
une partie squelettique $p$-calibrée fermée pure~$P_f$.
Supposons que la famille~$(P_f)$ vérifie les conditions suivantes:
\begin{enumerate}
\item Pour tout~$f$, l'application paralinéaire $f|_{P_f}$ est 
une immersion par morceaux;
\item On a $P_{f|W}=P_f\cap W$ pour toute carte tropicale~$f$
sur un ouvert~$V$ de~$X$ et tout ouvert~$W$ contenu dans~$V$;
\item S'il existe une application affine~$q$ tel que $f=q\circ g$,
alors $P_f=P_g^{(f)}$.
\end{enumerate}
Alors, il existe un unique courant squelettique~$T$ sur~$X$
tel que $P_f$ décrive~$T$ relativement à~$f$, pour toute carte
tropicale~$f$ sur un ouvert de~$X$.
\end{prop}
\begin{proof}
Soit tout d'abord un ouvert~$V$ de~$X$,
des cartes tropicales~$f,f'$ sur~$V$, des formes $\omega,\omega'$,
respectivement $f$-tropicales et $f'$-tropicales,
des fonctions lisses $\rho,\rho'$ à support compact sur~$V$
telles que $\rho\omega=\rho'\omega'$.
En posant $g=(f,f')$, on voit que 
\[ \int_{P_f} \rho\omega|_{P_f}
 = \int_{P_g} \rho\omega|_{P_g}
 = \int_{P_g} \rho'\omega'|_{P_g}
 = \int_{P_{f'}}\rho'\omega'|_{P_{f'}}, \]
en vertu de l'hypothèse~(3).

L'unicité d'un courant~$T$ vérifiant les propriétés requises
résulte de la remarque~\ref{rema.formes-test}.
Établissons son existence.

Soit $\omega$ une forme lisse à support compact sur~$X$.
Par définition, il existe un recouvrement fini $(U_i)$ du 
support de~$\omega$ par des parties ouvertes de~$X$ sur lesquelles
$\omega$ est tropicale. Soit $(\rho_i)$ une famille de fonctions
lisses à support compact 
dont la somme $\rho=\sum\rho_i$ vaut~$1$ au voisinage du support de~$\omega$
(de sorte que $\omega=\rho\omega$)
et telles que pour tout~$i$, le support de~$\omega_i$ soit contenu dans~$U_i$.
Pour tout~$i$, soit $f_i$ un moment sur~$U_i$ qui tropicalise~$\omega|_{U_i}$.
On veut poser
\[ \langle T,\omega\rangle = \sum_i \int_{P_{f_i}} \rho_i\omega|_{U_i}. \]
Vérifions d'abord que cette expression ne dépend pas du choix
de la famille~$(U_i,f_i,\rho_i)$.
Soit donc $(V_j,g_j,\theta_j)_{j}$ une autre famille finie vérifiant
les mêmes propriétés; posons $\theta=\sum_j\theta_j$.
Pour tout couple~$(i,j)$, la fonction $\rho_i\theta_j$ est lisse
et son support est une partie compacte de~$U_i\cap V_j$.
Puisque $ \omega=\theta\omega = \sum_j\theta_{j}\omega$, on a 
\[
\sum_i \int_{P_{f_i}} \rho_i\omega|_{U_i}
= \sum_{i,j} \int_{P_{f_i}} \rho_i \theta_j \omega|_{U_i}. \]
De même, on a 
\[ \sum_j \int_{P_{g_j}} \theta_j\omega|_{V_j} 
= \sum_{i,j} \int_{P_{g_j}} \rho_i \theta_j \omega|_{V_j}. \]
Fixons un couple~$(i,j)$.
Par hypothèse, $P_{f_i}\cap V_j = {P_{f_i|_{U_i\cap V_j}}}$
et ${P_{g_j}\cap U_i} = {P_{g_j|_{U_i\cap V_j}}}$; comme
$\rho_i\theta_j$ est à support compact contenu dans~$U_i\cap V_j$, on a donc
\[ \int_{P_{f_i}} \rho_i \theta_j \omega|_{U_i}
= \int_{P_{f_i|_{U_i\cap V_j}}} \rho_i\theta_j \omega|_{U_i\cap V_j} \]
et 
\[ \int_{P_{g_j}} \rho_i \theta_j \omega|_{V_j}
= \int_{P_{g_j|_{U_i\cap V_j}}} \rho_i\theta_j \omega|_{U_i\cap V_j} . \]
Soit $h$ le moment~$(f_i,g_j)$ sur~$U_i\cap V_j$.
La troisième propriété des parties~$P_f$ entraîne alors
\[  
\int_{P_{f_i|_{U_i\cap V_j}}} \rho_i\theta_j \omega|_{U_i\cap V_j} 
= \int_{P_h} \rho_i\theta_j \omega|_{U_i\cap V_j}
= \int_{P_{g_j|_{U_i\cap V_j}}} \rho_i\theta_j \omega|_{U_i\cap V_j} .\]
Cela entraîne l'égalité
\[ \sum_i \int_{P_{f_i}} \rho_i\omega|_{U_i} = \sum_j \int_{P_{g_j}} \theta_j\omega|_{V_j} \]
et légitime la définition de $\langle T,\omega\rangle$.

Il est immédiat que l'application $\omega\mapsto \langle T,\omega\rangle$
est linéaire.

Pour démontrer que c'est un courant, 
on applique le critère du lemme~\ref{lemm.continuite-courant}.
Soit $U$ un ouvert de~$X$, soit $f$ une carte tropicale sur~$U$ 
et soit $\rho$ une fonction lisse sur~$X$ dont le support est 
compact et contenu dans~$U$, 
Pour toute forme lisse~$\omega$ sur~$X$, 
tropicalisée par~$f$ sur~$U$, on a 
\[ \langle T, \rho\omega\rangle = \int_{P_f} \rho\omega|_{P_f}. \]
Le membre de droite provient de la forme linéaire $
\alpha \mapsto \int_{P_f} \rho|_{P_f}\alpha $
sur~$\mathscr A^{p,p}(P_f)$
qui est continue pour la topologie~$\mathscr C^0$.
Cela prouve que $T$ est un courant, et même un courant d'ordre~$0$,
et que c'est un courant paralinéaire décrit par~$P_f$ relativement à~$f$.
\end{proof}

\begin{coro}
Soit $X$ un espace tropical de niveau~$n$,
localement différentiellement séparé, et soit $p$ un entier naturel.
Les courants squelettiques de bidimension~$(p,p)$ sur~$X$
forment un treillis.
\end{coro}
\begin{proof}
Soit $S$ et $T$ des courants squelettiques de bidimension~$(p,p)$ sur~$X$.
Pour toute carte tropicale~$f$ sur un ouvert de~$X$,
posons $P_f^{S\wedge T}=P_f^S\wedge P_f^T$
et $P_f^{S\vee T}=P_f^S\vee P_g^T$
(lemme~\ref{lemm.squelette-calibre-treillis}).
Les familles
de parties squelettiques $p$-calibrées fermées $(P_f^{S\wedge T})_f$
et $(P_f^{S\vee T})_f$ 
définissent des courants squelettiques $S\wedge T$ et $S\vee T$
sur~$X$ tels que $S\vee T=\sup(S,T)$ et $S\wedge T=\inf(S,T)$.
\end{proof}

\subsection{}\label{ss.courants-squel-diff}
En particulier, tout courant squelettique~$T$ sur~$X$
se décompose comme la différence $T^+-T^-$ des deux courants
squelettiques positifs $T^+=\sup(0,T)$ et $T^-=\sup(0,-T)$.

\begin{rema}
Soit $T$ un courant squelettique de bidimension~$(p,p)$ sur~$X$
et soit $\omega$ une forme lisse sur un ouvert~$U$ de~$X$, 
tropicalisée par une carte tropicale~$f$ sur~$U$. 
Comme $T$ est différence de deux courants squelettiques positifs,
on dispose d'une mesure $\mu^T_\omega$ sur~$U$.
Par construction, elle applique toute fonction~$\rho$, lisse à support
compact sur~$U$, sur $\langle T,\rho\omega\rangle$.
Par définition d'un courant squelettique, on a donc
\[ \mu^T_\omega (\rho) = \int_{P_f^T} \rho\omega|_{P_f^T}. \]
Cette formule s'étend alors au cas où $\rho$ est borélienne.
\end{rema}

\begin{prop}\label{prop.paralineaire-PfT}
Soit $T$ un courant squelettique de dimension~$p$ sur~$X$,
soit $f$ une carte tropicale sur un ouvert~$U$ de~$X$, 
soit $q$ un entier naturel tel que $q\leq p$,
soit $\omega$ une forme  de type $(q,p)$ sur~$U$, $f$-tropicale, 
et soit $\alpha$ une forme de type $(p-q,0)$ sur~$U$
telle que $\alpha\wedge\omega$ soit à support compact sur~$U$.
Alors on a
\[ \langle T,\alpha\wedge\omega\rangle = \int_{P_f^T} \alpha\wedge\omega. \]
\end{prop}
On en déduit une formule analogue pour des formes~$\omega$ de type $(p,q)$
et des formes~$\alpha$ de type $(0,p-q)$.
\begin{proof}
Soit $(U_i)$ une famille finie d'ouverts de~$U$ recouvrant le support
de~$\alpha\wedge\omega$ et telle que pour tout~$i$,
la forme $\alpha$ soit tropicalisée sur~$U_i$ par une carte tropicale~$g_i$
sur~$U_i$. Soit $(\theta_i)$ une partition de l'unité 
au voisinage du support de~$\alpha\wedge\omega$ qui est subordonnée au
recouvrement~$(U_i)$. On a 
\[ \langle T, \alpha\wedge\omega\rangle
 = \sum_i \langle T, \theta_i \alpha \wedge\omega\rangle. \]
Cela ramène au cas où $\alpha$ est tropicalisée par une carte~$g$
et où la famille $(U_i)$ est réduite à l'ouvert~$U$.
Par définition d'un courant squelettique, on a alors
\[ \langle T, \alpha\wedge\omega\rangle = \int_{P_{(f,g)}^T} \alpha\wedge\omega.\]
Comme $\omega$ est de type~$(q,p)$, sa restriction à toute cellule
de~$P_{(f,g)}^T$ qui est contractée par~$f$ est nulle. Ainsi,
\[ \langle T, \alpha\wedge\omega\rangle = \int_{(P_{(f,g)}^T)^{(f)}} \alpha\wedge\omega.\]
On a vu au \S\ref{ss.PfT} que 
la partie squelettique $p$-calibrée $(P_{(f,g)}^T)^{(f)}$ 
est égale à $(P_{f}^T)^{(f,g)}$, c'est-à-dire à $P_f^T$ puisque, $f$
étant une immersion par morceaux sur $P_f^T$, 
l'application $(f,g)$ en est également une.
Ainsi,
\[ \langle T, \rho \alpha\wedge\omega\rangle
= \int_{P_f^T} \rho\alpha\wedge\omega ,\]
d'où la proposition.
\end{proof}

\section{Courants squelettiques et G-topologie}
\label{sec.courants-squelettiques-G}

Comme les courants squelettiques sont différence de courants
positifs (\S\ref{ss.courants-squel-diff}),
il résulte du paragraphe~\ref{sec.courants-positifs-G}
que leur action sur les formes s'étend aux G-formes 
(à support compact) sur des domaines
(non nécessairement ouverts) de~$X$.
Le but de ce paragraphe est de donner une description
squelettique de cette extension 
et de montrer que les courants squelettiques ainsi étendus
forment un G-faisceau.

\begin{lemm}\label{lemm.comparaison-polyedres-T}
Soit $X$ un espace tropical de niveau~$n$, localement différentiellement séparé.
Soit $T$ un courant squelettique de dimension~$p$ sur~$X$.
Soit $f$ et~$g$ deux cartes tropicales sur~$X$ et soit $V$
un domaine de~$X$ sur lequel $f$ et $g$ coïncident.
Alors $\abs{P_f^T\cap V}=\abs{P_g^T\cap V}$.
\end{lemm}
\begin{proof}
Par définition, $\abs{P_f^T\cap V}$ est la réunion des $p$-cellules
fermées de~$P_f^T$ qui sont contenues dans~$V$ et dont le calibrage
est partout non nul.
Soit $C$ une telle cellule;
prouvons qu'elle est contenue dans $P_g^T$, avec même calibrage.
Comme $C\subseteq P_f^T$, $f|_C$ est une immersion par morceaux;
puisque $C\subseteq V$, $g|_C=f|_C$ donc $g|_C$
est une immersion par morceaux. Par suite, $C\subseteq (P_f^T)^{(g)}$,
lequel est égal à $(P_g^T)^{(f)}$ comme partie squelettique $p$-calibrée
(proposition~\ref{prop.comparaison-polyedres-T});
en particulier, $C\subseteq P_g^T$, avec même calibrage.
Cela prouve l'inclusion 
$\abs{P_f^T\cap V}\subset \abs{P_g^T\cap V}$
de parties squelettiques $p$-calibrées.
Par symétrie, on a égalité.
\end{proof}

\emph{Dans toute la suite de ce paragraphe, nous supposons 
que l'espace tropical~$X$ est localement différentiellement séparé
et décent} (définition~\ref{defi.espace-decent}).

\begin{prop}\label{prop.courant-squ-G}
Soit $T$ un courant squelettique de dimension~$p$ sur~$X$.
Il existe une unique application
$(V,f)\mapsto P_f^T$, qui,  à tout domaine~$V$ de~$X$ 
et toute carte G-tropicale~$f$ sur~$V$ 
associe une partie squelettique $p$-calibrée sur~$V$, 
vérifiant les conditions suivantes:
\begin{enumerate}\def\theenumi{\roman{enumi}}\def\labelenumi{(\theenumi)}
\item
Si $U$ est un ouvert de~$X$ et $f$ une carte tropicale sur~$U$,
alors $P_f^T$ est 
la plus petite partie squelettique $p$-calibrée 
définissant~$T$ relativement à~$f$,
c'est-à-dire celle déjà noté ainsi;
\item
Si $V$ est un domaine de~$X$, $f$ une carte G-tropicale sur~$V$
et $W$ un domaine de~$V$, alors $P_{f|_W}^T=\abs{P_f^T\cap W}$.
\end{enumerate}
De plus, pour toute carte G-tropicale~$f$ sur un domaine~$V$ de~$X$,
la partie squelettique $p$-calibrée~$P_f^T$ est pure
et la restriction de~$f$ à~$P_f^T$ est une immersion par morceaux.
\end{prop}
\begin{proof}
Démontrons l'unicité d'une telle application.

On démontre d'abord que les conditions imposées déterminent
$P_f^T$ lorsque $V$ est un domaine de~$X$ et $f$ une carte tropicale sur~$V$,
et entraînent que $P_f^T$ est une partie
squelettique sur laquelle $f$ est une immersion par morceaux.

Il suffit de démontrer les autres propriétés localement sur~$V$.
Soit $x$ un point de~$V$. 
Comme $X$ est décent, il existe un voisinage ouvert~$U$
de~$x$ dans~$X$ et un carte tropicale~$g$ sur~$U$ tel que $f|_{U\cap V}
=g|_{U\cap V}$. 
La condition~(i) impose $P_{g|_{U\cap V}}^T=\abs{P_g^T\cap (U\cap V)}$,
où $P_g^T$ est la partie squelettique $p$-calibrée décrivant le courant~$T$
relativement  à la carte tropicale~$g$ sur l'\emph{ouvert}~$U$ de~$X$.
La condition~(ii) appliquée à~$V$ et~$U\cap V$ et à la carte tropicale~$f$
impose $P_{f|_{U\cap V}^T}= \abs{P_f^T \cap (U\cap V)}$;
appliquée à~$U$ et~$U\cap V$ et à la carte~$g$, elle impose
$P_{g|_{U\cap V}^T}= \abs{P_g^T \cap (U\cap V)}$.
Comme $U\cap V$ est un ouvert de~$V$, 
on a $\abs{P_f^T \cap (U\cap V)}
=\abs{P_f^T}\cap (U\cap V)$; 
la condition~(ii) appliquée à~$V$, $V$ et à la carte~$f$
entraîne que $P_f^T=\abs{P_f^T}$,
si bien que $\abs{P_f^T\cap (U\cap V)}=P_f^T \cap (U\cap V)$.
Ainsi, les conditions de la proposition entraînent
que 
\[ P_f^T \cap (U\cap V)=  \abs{P_g^T \cap (U\cap V) }, \]
donc est déterminé sur~$U\cap V$.

Comme la restriction de~$g$ à~$P_g^T$ est une immersion par
morceaux et que $f|_{U\cap V}=g|_{U\cap V}$,
cette formule montre également 
que la restriction de~$f$ à~$P_f^T\cap (U\cap V)$
est une immersion par morceaux.

Démontrons maintenant que les conditions imposées déterminent~$P_f^T$
lorsque $f$ est une carte G-tropicale sur un domaine~$V$ de~$X$
et entraînent que $P_f^T$ est pure, 
et que $f|_{P_f^T}$ est une immersion par morceaux.

La relation~(ii) appliquée à~$W=V$ entraîne que $P_f^T$ est pure.

Soit $(V_i)$ un G-recouvrement de~$V$ qui tropicalise~$f$.
Les parties $P_i=P_{f|_{V_i}}^T=\abs{P_f^T \cap V_i}$ 
sont déterminées par ce qui précède;
d'après la remarque~\ref{rema.polyedres-purs-G-faisceau},
cela détermine~$P_f^T$.
De plus, la restriction de~$f$ à~$P_f^T\cap V_i$ est une immersion par
morceaux pour tout~$i$, donc la restriction de~$f$ à~$P_f^T$ 
est une immersion par morceaux.

Établissons enfin l'existence d'une telle application.
Soit $(V,f)$ une carte G-tropicale.
Choisissons un G-recouvrement~$(V_i)$ de~$V$ tropicalisant~$f$.
Pour tout~$i$, choisissons une famille $(U_{j}, g_{j})_{j\in J_i}$
où $U_j$ est un ouvert de~$X$ et $g_j$ une carte tropicale sur~$U_j$ 
tels que $g_{j}|_{U_j\cap V_i}
= f|_{U_j\cap V_i}$, et telle que les $U_j$ recouvrent~$V_i$.
L'existence d'une telle famille découle de l'hypothèse faite sur~$X$.
Quitte à remplacer les~$V_i$ par $V_i\cap U_j$, on suppose
que $(V_i)$ et $(U_j)$ sont indexées par le même ensemble
et que $V_i\subset U_i$ pour tout~$i$.

Pour tout~$i$, définissons une  partie squelettique $p$-calibrée
pure de~$V_i$ par la formule
$ P_i = \abs{P_{g_i}^T\cap (U_i\cap V)}$.
Vérifions qu'elles se recollent au sens des parties
squelettiques $p$-calibrées pures, c'est-à-dire
que pour tous $i,j$,
on a
$\abs{P_i \cap V_j}=\abs{P_j\cap V_i}$,
soit encore
\[ \abs{P_{g_i}^T \cap (V_i \cap V_j)}=\abs{P_{g_j}^T\cap (V_i\cap V_j)}. \]
Sur le domaine analytique $V_i\cap V_j$ de~$U_i\cap U_j$, 
les cartes~$g_i$ et~$g_j$  coïncident avec~$f$;
l'égalité voulue découle donc du lemme~\ref{lemm.comparaison-polyedres-T}.
Soit $P$ la partie squelettique $p$-calibrée pure obtenue
par recollement des~$P_i$
(remarque~\ref{rema.polyedres-purs-G-faisceau}).
Vérifions qu'elle ne dépend pas de la famille $(V_i,U_i,g_i)$:
si $(V'_j,U'_j,g'_j)$ en est une autre, 
définissant une partie squelettique $p$-calibrée~$P'$,
la propriété de recollement appliquée à la réunion de ces deux
familles entraîne   que $P=P'$.
Notons-le $P_f^T$.

Il reste à vérifier les propriétés~(i) et~(ii).

 
La propriété~(i) est obtenue via la famille $(V,U,f)$.

La propriété~(ii) est obtenue en conservant la même famille $(V_i,U_i,g_i)$:
La partie squelettique $p$-calibrée pure $P_{f|_W^T}$ 
est obtenue par recollement des $\abs{P_{g_i}^T \cap (U_i\cap W)}$
qui est égal à $\abs{\abs{P_{g_i}^T \cap (U_i \cap V)} \cap W}$,
si bien que leur recollement est égal
à $\abs{ P_f^T \cap W}$.
\end{proof}

\begin{prop}
Soit $T$ un courant squelettique de dimension~$p$ sur~$X$,
soit $f$ une carte G-tropicale sur un domaine~$V$ de~$X$,
soit $\omega$ une G-forme sur~$V$ qui est tropicalisée par~$f$
et soit $\rho$ une fonction G-lisse sur~$V$
telles que $\rho\omega$ soit à support compact dans~$V$.
Alors 
\[ \langle T, \rho\omega\rangle = \int_{P^T_f} \rho\omega|_{P^T_f} . \]
\end{prop}
\begin{proof}
Soit $(V_i)$ une famille finie de domaines compacts de~$V$
recouvrant le support de~$\rho\omega$ et sur lesquels $\rho$ est tropicale
et $f$ est affine.
Pour toute partie finie non vide d'indices~$I$, posons $V_I=\bigcap_{i\in I}V_i$.
Par construction,
\[ \langle T,\rho\omega\rangle = \sum_{I\neq \emptyset}(-1)^{\Card(I)-1}
  \mu^T_{\omega}(\rho\mathbf 1_{V_I}) \]
et
\[ \int_{P_f^T} \rho|_{P_f^T} = \sum_{I\neq \emptyset}(-1)^{\Card(I)-1}
  \int_{P_f^T\cap V_I} \rho\omega|_{P^T_f\cap V_I}. \]
Il suffit de comparer ces deux expressions terme à terme, 
ce qui permet de supposer que $V$ est compact,
que $f$ est affine et que $\rho$ est lisse sur~$V$.
Il existe une forme~$\alpha'$ sur~$f(V)$ telle que $\omega=f^*\alpha'$;
comme $f(V)$ est compact, cette forme lisse s'étend en une forme lisse~$\alpha$
sur tout le codomaine de~$f$.
Choisissons une famille finie $(U_j, g_j)_{j\in J}$ recouvrant~$V$,
où $U_j$ est un ouvert de~$X$ et $g_j$ une carte tropicale sur~$U_j$
qui prolonge $f|_{V\cap U_j}$,
une fonction lisse~$\rho_j$ sur~$U_j$ qui prolonge~$\rho|_{V\cap U_j}$;
pour tout~$j$,
la forme lisse $g_j^*\alpha$ sur~$U_j$ prolonge~$\omega|_{V\cap U_j}$.
Soit $(\theta_j)$ une partition de l'unité au voisinage de~$V$
subordonnée au recouvrement~$(U_j)$.
Comme le support de~$\rho\omega$  est contenu dans~$V$, on a
$ \rho\omega= \sum_j \theta_j \rho\omega$, d'où
\[ \langle T,\rho\omega\rangle = \sum_{j\in J} \langle T, \theta_j\rho\omega\rangle\]
et
\[ \int_{P_f^T}\rho\omega|_{P_f^T} = \sum_{j\in J} \int_{P_f^T} \theta_j\rho\omega|_{P_f^T}. \]
Là encore, il suffit de comparer ces expressions terme à terme,
ce qui permet de supposer que $f$ et $\rho$
sont les restrictions à~$V$
d'une carte tropicale $g$ définie sur~$X$
et d'une fonction lisse (encore notée~$\rho$) sur~$X$.
Alors, $\omega$ est la restriction à~$V$ de la forme lisse~$g^*\alpha$ sur~$X$.
Par définition de l'extension d'un courant positif aux G-formes,
il s'agit de démontrer que 
$\mu^T_{g^*\alpha}(\rho\mathbf 1_V)= \int_{P^T_f} \rho \omega|_{P^T_f}$.
Par définition d'un courant squelettique,
pour toute fonction lisse~$u$ à support compact sur~$X$,
on a
\[ \mu ^T_{g^*\alpha} (u) = \langle T, ug^*\alpha\rangle
 = \int_{P^T_g} u g^*\alpha|_{P^T_g}. \]
Cette formule s'étend donc à toute fonction continue à support
compact, et même à toute fonction intégrable.
En particulier,
\[ \mu^T_{g^*\alpha}(\rho\mathbf 1_V)
 = \int_{P^T_g} \rho \mathbf 1_V g^*\alpha|_{P^T_g}.
  \]
D'après la proposition~\ref{prop.courant-squ-G},
on a $\abs{P^T_g\cap V} = P^T_f$, si bien que 
\[ \mu^T_\omega(\rho)=\int_{P^T_f} \rho f^*\alpha|_{P^T_f}, \]
comme il fallait démontrer
\end{proof}

\begin{prop}\label{prop.comparaison-polyedres-T-G}
Soit $T$ un courant squelettique 
de dimension~$p$ sur~$X$
Soit $f,g$ des cartes G-tropicales sur~$X$. On a l'égalité
de parties squelettiques $p$-calibrées  pures:
$ (P_f^T)^{(g)} = (P_g^T)^{(f)} $.
\end{prop}
\begin{proof}
On peut prouver cette relation G-localement.
Il suffit donc de démontrer que pour tout domaine~$V$ de~$X$,
tout ouvert~$U$ de~$X$ dans lequel $V$ est fermé 
et tel que
$f|_V$ et~$g|_V$ sont les restrictions de cartes tropicales~$f'$ et~$g'$ 
sur~$U$, on a 
\[ \abs{(P_f^T)^{(g)} \cap V} = \abs{(P_g^T)^{((f)} \cap V}. \]
Par construction, $\abs{P_f^T\cap V}=\abs{P_{f'}^T\cap V}$.
Alors,
\[ \abs{(P_f^T)^{(g)} \cap V}
 = \abs{P_f^T \cap V}^{(g)}
 = \abs{P_{f'}^T\cap V}^{(g)}
 = \abs{(P_{f'}^T)^{(g')}\cap V}. \]
De même,
\[  \abs{(P_g^T)^{(f)} \cap V}
 = \abs{(P_{g'}^T)^{(f')}\cap V}. \]
D'après la proposition~\ref{prop.comparaison-polyedres-T}
appliquée aux cartes tropicales~$f'$ et~$g'$ sur~$U$,
on a $(P_{g'}^T)^{(f')}= (P_{f'}^T)^{(g')}$; la proposition en résulte.
\end{proof}

\begin{rema}\label{rema.G-courant-para-polyedre}
Inversement, 
considérons une famille~$(P_f)$ de parties squelettiques $p$-calibrées
pures, indexée par les cartes G-tropicales de~$X$
et satisfaisant les conditions suivantes:
\begin{enumerate}\def\theenumi{\roman{enumi}}\def\labelenumi{(\theenumi)}
\item
Pour toute carte G-tropicale~$f$, la restriction de~$f$ à~$P_f$
est une immersion par morceaux.
\item
Si $V$ est un domaine de~$X$, $f$ une carte G-tropicale sur~$V$
et $W$ un domaine de~$V$, alors $P_{f|_W}=\abs{P_f\cap W}$.
\item
S'il existe une application affine~$q$ telle que $f=q\circ g$,
alors $P_f=P_g^{(f)}$.
\end{enumerate}

D'après la remarque~\ref{rema.courant-para-polyedre},
il existe un unique courant squelettique~$T$ sur~$X$
tel que $P_f^T=P_f$ pour toute carte tropicale~$f$ définie sur un ouvert de~$X$.

Il résulte de la construction que l'on a
$P_f^T=P_f$ pour toute carte G-tropicale sur un domaine~$V$ de~$X$.
\end{rema}

\subsection{}
Soit $T$ un courant squelettique de dimension~$p$ sur~$X$.
Soit~$V$ un domaine de~$X$; observons que $V$ est décent
et localement différentiellement séparé.
La famille des parties squelettiques
$p$-calibrées $(P_f^T)$, lorsque $f$ parcourt les cartes G-tropicales
définies sur un domaine de~$V$ vérifie les conditions
du paragraphe~\ref{rema.G-courant-para-polyedre}.
Elle définit donc une forme linéaire sur les G-formes de~$V$
dont la restriction aux formes  est un courant squelettique
que nous noterons~$T|_V$.

On définit ainsi un G-préfaisceau sur~$X$ qui prolonge le
faisceau des courants squelettiques.

\begin{lemm}
Ce G-préfaisceau est un G-faisceau.
\end{lemm}
\begin{proof}
Soit $V$ un domaine de~$X$ et soit $(V_i)_{i\in I}$ un G-recouvrement
de~$V$; pour tout~$i$, soit $T_i$ un courant squelettique
de dimension~$p$ sur~$V_i$  et
supposons que l'on a $T_i|_{V_i\cap V_j}=T_j|_{V_i\cap V_j}$
pour tous~$i,j$. Il s'agit de démontrer
qu'il existe un unique courant squelettique~$T$ de dimension~$p$
sur~$V$ tel que $T_i=T|_{V_i}$ pour tout~$i$.
Quitte à remplacer~$X$ par~$V$, on suppose que $V=X$ pour alléger les notations.

Suivant la remarque~\ref{rema.courant-para-polyedre},
il s'agit de définir des parties squelettiques $p$-calibrées~$P_f$ indexées
par les cartes tropicales~$f$ définies sur un ouvert~$U$ de~$X$
vérfiant les conditions de cette remarque.
 
Fixons $U$ et~$f$. On définit la partie squelettique calibrée~$P_f$ 
de~$U$ G-localement par rapport au G-recouvrement $(V_i\cap U)$.
Soit $i,j$ des indices dans~$I$;
par définition du courant $T_i|_{U\cap V_i\cap V_j}$,
on a 
\[ P^{T_i|_{U\cap V_i\cap V_j}}_{f}
= \abs{P^{T_i}_{f} \cap (U\cap V_i\cap V_j)}. \]
Comme les courants $T_i$ et $T_j$ coïncident sur $V_i\cap V_j$,
on a 
\[ P^{T_i|_{U\cap V_i\cap V_j}}_{f}
= P^{T_j|_{U\cap V_i\cap V_j}}_{f}. \]
Enfin, on a 
\[ P^{T_j|_{U\cap V_i\cap V_j}}_{f}
= \abs{P^{T_j}_{f} \cap (U\cap V_i\cap V_j)}. \]
Par suite, 
on a l'égalité 
\[  \abs{P^{T_i}_{f} \cap (U\cap V_i\cap V_j)}
=  \abs{P^{T_j}_{f} \cap (U\cap V_i\cap V_j)}. \]
Par recollement (remarque~\ref{rema.polyedres-purs-G-faisceau}, 
on en déduit qu'il existe
une unique partie squelettique $p$-calibrée~$P_{f}$
sur $U$ telle que pour tout $i\in I$, on ait
\[ \abs{P_{f} \cap (U\cap V_i)}
 = \abs{P^{T_i}_f \cap (U\cap V_i)}. \]

Il reste à démontrer que ces parties squelettiques $p$-calibrées~$P_f$ 
vérifient les conditions de la remarque~\ref{rema.courant-para-polyedre}:

a) L'application paralinéaire~$f$ est une immersion 
par morceaux sur~$P_f$ car c'en est-une sur chaque $P^f\cap V_i\subset P_i$;

b) Si $V$ est un ouvert de~$U$, on a $P_{f|_V}=P_f\cap V$.
Il suffit de le vérifier sur chaque~$V_i$, où cela résulte
de la condition correspondante pour $P_f^{T_i}$.

c) Si $q$ est une application affine telle que $f=q\circ g$,
la condition $\abs{P_f}=\abs{P_g}^{(f)}$ 
se vérifie G-localement en restriction à chaque~$V_i$,
et elle résulte  alors de la condition correspondante
pour $P_f^{T_i}$.

Pour conclure la démonstration du lemme, 
le courant squelettique~$T$ associé à ces parties squelettiques $p$-calibrées
vérifie $T|_{V_i}=T_i$ pour tout~$i$,
par construction, et c'est le seul.
\end{proof}

\begin{prop}\label{prop.paralineaire-PfT-G}
Soit $T$ un courant squelettique de dimension~$p$ sur~$X$,
soit $f$ une carte G-tropicale sur un domaine~$U$ de~$X$, 
soit $q$ un entier naturel tel que $q\leq p$,
soit $\omega$ une G-forme G-lisse de type $(q,p)$ sur~$U$, 
$f$-tropicale, 
soit $\alpha$ une G-forme G-forme de type $(p-q,0)$
telle que $\alpha\wedge\omega$ soit 
à support compact sur~$U$.
Alors on a
\[ \langle T, \alpha\wedge\omega\rangle = \int_{P_f^T} \alpha\wedge\omega. \]
\end{prop}
On en déduit une formule analogue pour des G-formes~$\omega$ de type $(p,q)$
et des G-formes~$\alpha$ de type $(0,p-q)$.
\begin{proof}
Par construction, $\alpha\wedge\omega$ est une G-forme à support
compact sur~$X$; reprenons la construction  
de $\langle T,\alpha\wedge\omega\rangle$.
Soit $(V_i)$ une famille finie de domaines compacts
recouvrant le support de~$\alpha\wedge\omega$ et
telle que pour tout~$i$, la forme $\alpha|_{V_i}$ soit G-tropicalisée
par une carte G-tropicale~$g_i$. Compte tenu des définitions,
on se ramène à supposer que $U=V_i$ et que $\alpha$ est 
G-tropicalisée par une carte G-tropicale~$g$.
Alors, on a 
\[ \langle T, \alpha\wedge\omega\rangle
= \int_{P_{(f,g)}^T}  \alpha \wedge\omega. \]
Comme $\omega$ est de type $(q,p)$, sa restriction 
à toute cellule de~$P_{(f,g)}^T$ qui est contractée par~$f$
est nulle. Ainsi, 
\[ \langle T,  \alpha\wedge\omega\rangle
= \int_{(P_{(f,g)}^T)^{(f)}}  \alpha \wedge\omega. \]
D'après la proposition~\ref{prop.courant-squ-G},
la partie squelettique $p$-calibrée $(P_{(f,g)}^T)^{(f)}$ 
est égale à $(P_{f}^T)^{(f,g)}$, c'est-à-dire à $P_f^T$ puisque, $f$
étant une immersion par morceaux sur $P_f^T$, 
l'application $(f,g)$ en est également une.
Ainsi,
\[ \langle T,  \alpha\wedge\omega\rangle
= \int_{P_f^T} \alpha\wedge\omega ,\]
d'où la proposition.
%
\end{proof}

\begin{prop}\label{prop.paralineaire-atlas}
Soit $T$ un courant squelettique de dimension~$p$ sur~$X$
et soit~$F$ une partie fermée de~$X$.
Soit $(U_i)$ un G-recouvrement de~$X$
par des domaines ;
on fait l'une des deux hypothèses suivantes:
\begin{enumerate}\def\theenumi{\roman{enumi}}\def\labelenumi{(\theenumi)}
\item Le recouvrement est localement fini et les~$U_i$ sont fermés;
\item La partie~$F$ possède un voisinage ouvert paracompact dans~$X$.
\end{enumerate}
Pour tout~$i\in I$, soit $f_i$ une carte G-tropicale sur~$U_i$.
Il existe alors une partie squelettique $p$-calibrée~$P$ pure de~$X$
telle que pour tout entier~$r$ tel que $0\leq r\leq p$,
toute G-forme~$\omega$ de type~$(p-r,p)$ sur~$X$
telle que $\omega|_{U_i}$ soit $f_i$-tropicale  pour tout~$i$,
et toute G-forme~$\rho$ sur~$X$ 
telle que $\supp(\rho\wedge\omega)\subset F$,
on ait $\langle T,\rho\wedge\omega\rangle
= \int_P \rho\wedge\omega$.
\end{prop}
Par définition, toute partie compacte d'un espace G-tropical
possède une base de voisinages ouverts dénombrables à l'infini,
de sorte que la proposition s'applique lorsque $F$ est compacte. 

\begin{proof}
Commençons par traiter le cas où $(U_i)$ est un G-recouvrement
localement fini de~$F$ par des domaines fermés 
et soit $U$ leur réunion.
Pour tout ensemble fini non vide~$I$ d'indices,
on a posé $U_I=\bigcap_{i\in I}U_i$ 
et notons~$f_I$ une carte G-tropicale sur~$U_I$ 
qui est plus fine que chacune des~$f_i|_{U_I}$ pour $i\in I$ 
(par exemple leur concaténation), de sorte que pour tout~$i\in I$,
toute forme sur~$U_I$ qui est $f_i$-tropicale est $f_I$-tropicale. 
Choisissons pour partie squelettique~$P$
la réunion (localement finie) des
parties squelettiques fermées~$P_{f_I}^T$. Pour tout~$I$,
le calibrage de~$P_{f_I}^T$ induit un calibrage de~$P$,
nul hors de~$P_{f_I}^T$; munissons~$P$ de
la somme alternée (localement finie) de ces calibrages,
celui de $P_{f_I}^T$ étant multiplié par $(-1)^{\Card(I)-1}$.
Démontrons que $\abs P$  convient.

Considérons donc une G-forme~$\omega$ de type $(p-r,p)$
telle que $\omega|_{U_i}$ soit G-tropicale pour tout~$i$,
et une G-forme~$\rho$ de type~$(r,p)$
telles que $\supp(\rho\wedge\omega)$ soit compact et contenu dans~$F$.
En écrivant  l'égalité de G-formes à coefficients tropicalement mesurables
\[ \rho\wedge\omega = \mathbf 1_U \rho\wedge\omega = \sum_{\emptyset\neq I}
(-1)^{\Card(I)-1} \mathbf 1_{U_I} \rho\wedge\omega, \]
  il vient 
\[ \langle T,\rho\wedge\omega\rangle
 = \sum_{\emptyset\neq I} (-1)^{\Card(I)-1}
  \langle T|_{U_I}, \rho\wedge\omega |_{U_I} \rangle . \]
Pour tout~$I$, la G-forme~$\omega|_{U_I}$ est $f_I$-tropicale
et est de type $(q,p)$, 
la proposition~\ref{prop.paralineaire-PfT} entraîne que l'on a
\[ \langle T,\rho\wedge\omega \rangle
=  \sum_{\emptyset\neq I} (-1)^{\Card(I)-1}
\int_{P_{f_I}^T} \rho\wedge \omega.  \]
On a donc $\langle T,\rho\wedge\omega\rangle=\int_P \rho\wedge\omega$,
d'où la proposition dans le cas considéré.

Démontrons maintenant la proposition
sous l'hypothèse que $F$ possède un voisinage ouvert paracompact~$U$.
Quitte à considérer le courant paralinéaire $T|_U$ sur~$U$, on peut
supposer que $X$ est paracompact.
Pour tout $x\in F$, il existe un voisinage domanial 
compact~$W_x$ de~$x$ dans~$X$  et un ensemble fini d'indices~$I_x$
tel que $x\in U_i$ et tel que
$W_x\cap U_i$ soit fermé dans $W_x$
pour tout $i\in I_x$, et tel
que la réunion des~$U_i$ pour $i\in I_x$ contienne~$W_x$.

D'après le lemme~\ref{lemm.app-paracompact},
il existe un recouvrement localement fini~$(W_j)_{j\in J}$ de~$X$
par des domaines compacts de~$X$, qui raffine
le recouvrement ouvert $(U_i)_{i\in I}$.
et tel que les intérieurs~$\mathring Z_j$ recouvrent~$X$.
Pour tout~$j\in J$, on choisit un indice~$i$
tel que $W_j\subset U_i$ et on note $g_j$
la carte G-tropicale~$f_i|_{W_j}$.

On est ainsi ramené au cas où l'hypothèse~(i) est satisfaite,
ce qui conclut la démonstration.
\end{proof}

\section{Calcul différentiel sur les courants}
\label{ss.calcul-diff-courants}

Soit $X$ un espace tropical de niveau~$n$.

\subsection{}
L'application linéaire~$\di$ de~$\mathscr A^{p,q}_\cpct(X)$
dans~$\mathscr A^{p+1,q}_\cpct(X)$ est continue
(prop.~\ref{prop.continuite-di-dc-J}).
On note~$\mathrm b'$ sa transposée
de~$\mathscr D_{p+1,q}(X)$ dans~$\mathscr D_{p,q}(X)$.
Si $T$ est un courant d'ordre~$\leq r$, alors $\mathrm b' T$
est un courant d'ordre~$\leq r+1$.

On définit de même~$\mathrm b''\colon\mathscr D_{p,q+1}(X)\ra\mathscr D_{p,q}(X)$ comme la transposée de l'application~$\dc$.

Il résulte de la définition de la topologie des espaces de courants
que ces applications linéaires~$\mathrm b'$ et $\mathrm b''$ sont continues.

Par dualité, la relation $\dc\omega=\mathrm J\di\mathrm J^{-1}\omega$ 
entraîne l'égalité
\begin{equation}
\mathrm b''=\mathrm J \mathrm b' \mathrm J^{-1}.
\end{equation}

On a aussi les relations
$(\mathrm b')^2=0$, $(\mathrm b'')^2=0$,
et $\mathrm b'\mathrm b''=-\mathrm b''\mathrm b'$.

%

\subsection{}
Soit~$\omega$ une G-forme de type~$(p,q)$ sur~$X$ 
et soit~$[\omega]$ le courant de bidegré~$(p,q)$ qui lui est associé. 
Pour toute forme~$\alpha\in\mathscr A^{n-p-1,n-q}_\cpct(X)$,
on a
\[ \di(\omega\wedge\alpha) 
      = \di \omega \wedge\alpha + (-1)^{p+q}\omega\wedge\di\alpha. \]
La formule de Stokes (proposition~\ref{prop.stokes}) entraîne alors
l'égalité
\[ \int^\partial_{X}\omega \wedge\alpha =  \int_X \di(\omega\wedge\alpha)=
    \int_X \di\omega \wedge\alpha 
   + (-1)^{p+q} \int_X \omega\wedge\di\alpha,\]
autrement dit
\[ [\di\omega] =  [\omega]^\partial +(-1)^{p+q+1} b'[\omega]. \]
Si l'on définit l'opérateur~$\di$ de~$\mathscr D^{p,q}(X)$
dans~$\mathscr D^{p+1,q}$ par
\begin{equation} \di = (-1)^{p+q+1} \mathrm b', \end{equation}
il vient alors
\begin{equation}
\di {[\omega]} =  [\di\omega] - [\omega]^\partial . \end{equation}
Lorsque $X$ est sans bord et $\omega$ est une \emph{forme} sur~$X$,
le courant $[\omega]^\partial$ est nul, et l'on a
\begin{equation}
\di {[\omega]}= [\di\omega].
\end{equation}

\subsection{}
De même, on pose 
\begin{equation}  \dc = (-1)^{p+q+1} \mathrm b''. \end{equation}
On en déduit que les relations du lemme~\ref{lemm.d'-d''-J} 
sont encore valides sur les espaces de courants:
\begin{gather}
\dc=\mathrm J \di \mathrm J^{-1} = -\mathrm J^{-1}\di\mathrm J \\
(\di)^2=(\dc)^2=0 \\
\dc\di=-\di\dc \\
\ddc\mathrm J=\mathrm J\ddc.
\end{gather}

Soit $\omega$ une G-forme de type~$(p,q)$. On a 
\begin{align*}
\dc {[\omega]} & = \mathrm J \di \mathrm J^{-1} [\omega]
 =  \mathrm J \di {[\mathrm J^{-1}\omega]}
 =  \mathrm J [\di\mathrm J^{-1}\omega]
-  \mathrm J [\mathrm J^{-1}\omega]^\partial \\
& = [\mathrm J\di\mathrm J^{-1}\omega]
 - \mathrm J [ \mathrm J^{-1}\omega]^\partial , 
\end{align*}
de sorte
que 
\begin{equation}
\dc{[\omega]}=[\dc\omega] - \mathrm J [\mathrm J^{-1}\omega]^\partial.
\end{equation}
En outre, si $X$ est sans bord et $\omega$ est une forme, 
alors $\mathrm J\omega$ est également une forme, 
donc on a 
\begin{equation}
\dc {[\omega]}=[\dc\omega].
\end{equation}

\subsection{}
Supposons que l'espace~$X$ est sans bord.
Comme le courant d'intégration~$\delta_X$ est égal à~$[1]$,
il est $\di$ et $\dc$-fermé.


\begin{lemm}
Pour tout courant~$S\in\mathscr D_{p,q}(X)$
et toute forme~$\alpha\in\mathscr A^{p,q}_\cpct(X)$,
on a
\[
\langle \ddc S,\alpha\rangle = \langle S, \ddc\alpha\rangle.
\]
\end{lemm}
\begin{proof}
En effet,
$ \ddc S =((-1)^{p+q+2}\bi) ((-1)^{p+q+1} \bc)S= - \bi\bc S$
est l'image de~$S$ par le transposé de~$- \dc\di=\ddc$.
\end{proof}

\begin{exem}
Soit $X$ un espace espace tropical de niveau~$n$.
Soit $\alpha$ une G-forme de type~$(p,q)$ 
à coefficients tropicalement de classe~$\mathscr C^2$.
On a donc
\begin{equation}
\ddc[\alpha] 
 =  \di \left([\dc\alpha]-\mathrm J[\mathrm J^{-1}\alpha]^\partial\right) 
=  [\ddc\alpha]-[\dc\alpha]^\partial-\di \mathrm J[\mathrm J^{-1}\alpha]^\partial.
\end{equation}

Supposons de plus que $\alpha$  soit symétrique.
Alors, $p=q$ et $\mathrm J^{-1}\alpha=\alpha$, 
de sorte que
\begin{equation}
\ddc[\alpha]
=  [\ddc\alpha]-[\dc\alpha]^\partial-\di \mathrm J[\alpha]^\partial.
\end{equation}

Supposons de plus que $X$ soit sans bord et $p=0$,
c'est-à-dire que $\alpha$
soit une fonction, G-localement tropicalisée par une fonction
de classe~$\mathscr C^2$, et démontrons que l'on a  
\begin{equation}
\ddc[\alpha] = [\ddc\alpha] - [\dc\alpha]^\partial.
\end{equation}
Compte tenu de la relation précédente, il suffit de vérifier que 
le courant $[\alpha]^\partial$ est nul.
Or, pour toute forme $\beta$ de type~$(n-1,n)$, 
lisse et à support compact, le nombre réel
$\langle [\alpha]^\partial,\beta\rangle$, égal à $\int^\partial_X \alpha\beta$,
est l'intégrale de~$\alpha$ contre la mesure de bord associée
à~$\beta$. Comme $\beta$ est une forme et comme $X$ est sans bord,
cette mesure de bord est nulle, d'où l'égalité
$[\alpha]^\partial=0$.

On retrouve en particulier
la formule classique définissant le laplacien combinatoire
d'un graphe (\cite{rumely1989}, \cite{zhang93}).
\end{exem}

\begin{lemm}
Soit $T$ un courant de bidimension~$(p,q)$
sur~$X$ et soit $\alpha$ une forme de type $(p',q')$.
On a l'égalité:
\[
\di(T\wedge \alpha) = (\di T)\wedge\alpha + (-1)^{p+q} T\wedge\di\alpha.
\]
\end{lemm}
\begin{proof}
En effet, pour toute forme~$\beta$ de type~$(p-p',q-q')$ sur~$X$, on a
\begin{align*}
\langle \di(T\wedge\alpha),\beta\rangle & = 
   (-1)^{p-p'+q-q'+1} \langle T\wedge\alpha ,\di\beta \rangle \\
& = (-1)^{p-p'+q-q'+1} \langle T,\alpha\wedge \di\beta \rangle \\
& = (-1)^{p-p'+q-q'+1} \langle T,(-1)^{p'+q'}(\di(\alpha\wedge \beta)
 - \di\alpha\wedge \beta) \rangle \\
& = (-1)^{p+q+1} \langle T,\di(\alpha\wedge \beta)\rangle
 +(-1)^{p+q} \langle T, \di\alpha\wedge \beta \rangle \\
& =  \langle \di T\wedge \alpha, \beta\rangle
 +(-1)^{p+q} \langle T\wedge \di\alpha,  \beta \rangle,
\end{align*}
d'où l'égalité voulue.
\end{proof}

\subsection{}
Soit $U$ un ouvert de~$X$. Pour tout courant~$T$ sur~$X$,
on a $\di(T|_U)=(\di T|_U)$ et $\dc(T|_U)=(\dc T)|_U$.

Soit $Y$ un espace tropical de niveau~$m$
et soit $\phi\colon Y\to X$ un morphisme compact
d'espaces tropicaux.
Pour tout courant~$T$ sur~$Y$, on a $\di(\phi_*T)=\phi_*(\di T)$
et $\dc(\phi_*T)=\phi_*(\dc T)$.

\begin{prop}\label{prop.deltaY-ferme}
Soit $Y$ un espace tropical de niveau~$m$ 
et soit $\phi\colon Y\to X$ un morphisme compact d'espaces tropicaux.
On a $\di (\phi_*\delta_Y) = \phi_*(\delta^\partial _{Y})$.
En particulier, si $X$ est sans bord,
le courant~$\delta_X$ est $\di$-fermé.
\end{prop}
\begin{proof}
Soit $d=\dim(Y)$ et soit $\alpha$ une $(d-1,d)$-forme lisse 
à support compact sur~$X$.
On a
\begin{align*}
 \langle \di(\phi_*\delta_Y),\alpha \rangle 
& = - \langle \phi_*\delta_Y,\di\alpha\rangle 
  = - \int_Y \phi^*\di\alpha \\
&  = - \int_Y \di\phi^*\alpha 
  = - \int^\partial_{Y} \phi^*\alpha  \\
&  = - \langle \phi_*\delta^{\partial}_Y,\alpha\rangle,
\end{align*}
d'où la formule annoncée.
\end{proof}

\begin{exem}\label{exem.d'-paralineaire}
Soit $T$ un courant squelettique de dimension~$p$ sur~$X$.
Soit $f$ une carte G-tropicale sur un domaine~$V$ de~$X$ et soit $P^T_f$
la partie squelettique $p$-calibrée pure
qui décrit~$T$ relativement à~$f$.
Soit $q$ un entier tel que $0\leq q\leq p-1$.
Démontrons que pour toute
G-forme~$\omega$ de type~$(q,p)$ sur~$V$ qui est $f$-tropicale,
et toute G-forme G-lisse~$\rho$
de type $(p-q-1,0)$ sur~$V$ telle que $\rho\wedge\omega$ soit
à support compact,  on a
\[ \langle T,\di(\rho\wedge\omega)\rangle 
= \int_{P_f^T}^\partial \rho\wedge\omega. \]

On a $\di(\rho\wedge\omega)=\di\rho\wedge\omega +(-1)^{p-q-1}\rho\wedge\di\omega$.
Comme les formes $\omega$ et $\di\omega$ sont de type $(q,p)$ et $(q+1,p)$
respectivement, 
la proposition~\ref{prop.paralineaire-PfT-G} entraîne les relations
\[ \langle T,\rho\wedge \di\omega \rangle = \int_{P_f^T} \rho\wedge \di\omega \]
et
\[ \langle T,\di\rho\wedge\omega \rangle = \int_{P_f^T} \di\rho\wedge\omega. \]
Ainsi,
\[ \langle T,\di(\rho\wedge\omega) \rangle
= \int_{P_f^T} (\di\rho\wedge\omega+(-1)^{p-q-1}\rho\wedge \di\omega)
= \int_{P_f^T} \di(\rho\wedge\omega) . \]
La formule de Stokes entraîne alors 
\[ \langle T,\di(\rho\wedge \omega) \rangle
= \int_{P_f^T}^\partial \rho\wedge\omega . \]

Lorsqu'on suppose que $f$ est une carte tropicale et que $\rho$ et $\omega$
sont des formes lisses, on a également
\[ \langle T, \di(\rho\wedge\omega)\rangle = 
- \langle \di T, \rho\wedge\omega\rangle.\]
La formule précédente se récrit alors
\[\langle \di T, \rho\wedge\omega \rangle
= -  \langle T, \di (\rho\wedge\omega)\rangle 
= - \int_{P_f^T}^\partial \rho\wedge\omega.\]
Informellement, le bord du courant squelettique~$T$ se calcule,
au signe près, en intégrant
les formes sur les bords des parties squelettiques calibrées associées à~$T$.
\end{exem}

\begin{rema}
La fermeture du courant~$\delta_X$, qui résulte
de l'absence de bord pour~$X$ (proposition~\ref{prop.deltaY-ferme}),
se traduit tropicalement
par des annulations de discordance (condition d'harmonie,
théorème~\ref{theo.condition-harmonie}).
Esquissons comment le résultat s'étend à un courant squelettique fermé~$T$.

Considérons une carte tropicale~$f$ définie sur un ouvert~$U$,
à valeurs dans~$\R^N$.
Soit $x$ un point de~$P_f^T$. 
Comme au \S\ref{defi.cone-tropical}, on définit un germe de cône
calibré~$\Pi_{f,x}^T$ dans~$\R^N$ au voisinage de~$f(x)$;
c'est l'image par l'immersion par morceaux~$f$ du germe 
de partie squelettique $p$-calibrée~$P_{f,x}^T$ en~$x$.
On démontre comme pour la condition d'harmonie que
ce germe~$P_{f,x}^T$ est harmonieux.
\end{rema}

\section{Cartes harmoniques}

Soit $X$ un espace tropical de niveau~$n$, sans bord.

\subsection{}
Soit $u\colon X\to\R^d$ une carte G-tropicale,
soit $P$ une partie paralinéaire de~$\R^d$ contenant~$u(X)$
et soit $\alpha$ une forme lisse sur~$P$.
Considérons la G-forme $\omega=u^*\alpha$.

Si $\omega$ est une forme, les courants $\di{[\omega]}$
et $[\di\omega]$ coïncident, mais ce n'est pas le cas en général,
leur différente 
est le courant de bord~$[\omega]^\partial$.
Nous allons montrer que si $u$ est « harmonique », ce phénomène
disparaît pour les G-formes du type~$u^*\alpha$.

\begin{defi}
On dit qu'une fonction paralinéaire~$u$ sur~$X$
est \emph{harmonique}\index{fonction paralinéaire harmonique}
si le courant $\ddc[u]$ est nul.
\end{defi}

Plus généralement, une fonction~$u$ sur~$X$ est dite harmonique
si elle est localement combinaison $\R$-linéaires de fonctions 
paralinéaires harmoniques.
Les fonctions harmoniques forment un sous-faisceau en $\R$-espaces vectoriels
du faisceau des fonctions continues sur~$X$.
Il contient les fonctions affines.

\begin{rema}
Du point de l'analyse fonctionnelle, il n'est pas clair que cette définition
soit optimale.  Contrairement à la situation en analyse complexe,
il n'y a aucune raison dans le contexte général des espaces tropicaux
qu'une fonction continue~$u$ telle que $\ddc[u]=0$
soit harmonique au sens précédent.

D'autre part, si $u$ est simplement supposée localement intégrable,
il lui est associée un courant~$[u]$ mais nous ne savons
pas si la condition $\ddc[u]=0$ entraîne même la continuité de~$u$
(à modification près qui ne change pas le courant).
\end{rema}


\begin{prop}\label{prop.int-bord-harmonique-zero}
Soit $u=(u_1,\dots,u_d)$ une suite finie de fonctions
paralinéaires harmoniques
sur~$X$. Soit $P$ un sous-espace paralinéaire de~$\R^d$
contenant~$u(X)$ et soit $\alpha$ une forme lisse 
de type~$(n-1,n)$ sur~$P$.
Soit $W$ un domaine de~$X$.
Pour toute fonction lisse~$\rho$ sur~$W$
telle que le support de $\rho u^*\alpha$ 
ne rencontre pas~$\partial(W)$,
on a $\int_W^\partial \rho u^*\alpha=0$.
\end{prop}
\begin{proof}
Soit $\Sigma$ une partie squelettique de~$W$
qui est un support fort de~$\rho u^*\alpha$
(proposition~\ref{lemm.squelette-support-fort}).
Soit $\mathscr C$ une décomposition cellulaire de~$\Sigma$
qui est adaptée au bord~$\partial(W)$, 
au calibrage de~$X$,
et telle les~$u_i$ soient affines sur chaque cellule.
Pour toute cellule~$C$ de dimension~$n$, notons~$\mu_C$
son calibrage.

On a
\[
 \int_W^\partial \rho u^*\alpha
 = \sum_{D\in\mathscr C_{n-1}}
\int_D \left(
  \sum_{C\supset D} \langle (\rho u^*\alpha)_C, \mu_C\rangle\right). \]

Soit $D$ une $(n-1)$-cellule. Si $D$ rencontre~$\partial(W)$,
chaque terme est nul, 
car le support de~$\rho u^*\alpha$ ne rencontre pas~$\partial(W)$.

Supposons désormais que $D$ ne rencontre pas~$\partial(W)$.
Si $C$ est une cellule contenant~$D$ telle que l'application
affine~$u|_C$ ne soit pas injective, on a $u^*\alpha=0$ sur~$C$.
Si~$\mathscr C_D$ est l'ensemble des cellules~$C$ de dimension~$n$
qui contiennent~$D$ telle que l'application affine~$u|_C$
soit injective, on a donc
\[
 \int_W^\partial \rho u^*\alpha
 = \sum_{D\in\mathscr C_{n-1}}
\int_D \left(
  \sum_{C\in \mathscr C_D} \langle (\rho u^*\alpha)_C, \mu_C\rangle\right). \]
Si $u|_D$ n'est pas injective, alors $u|_C$ n'est injective 
pour aucune cellule~$C$ contenant~$D$, donc $\mathscr C_D$ est vide.

Supposons que $u|_D$ soit injective et notons~$u_D^{-1}$
la bijection réciproque de $u_D\colon D\to u(D)$.
Soit  $C\in\mathscr C_D$;
notons $u_C^{-1}$ la bijection
réciproque de~$u|_C \colon  C \to u(C)$; sa restriction à~$u(D)$
coïncide avec~$u_D^{-1}$.
Sur $D$, on a l'égalité de $(n-1,0)$-formes 
\[ \langle (\rho u^*\alpha)_C, \mu_C \rangle
= \rho  u_C^* \langle \alpha_{u(C)} , u_*\mu_C\rangle ,  \]
donc 
\[ \int_D \langle (\rho u^*\alpha)_C, \mu_C \rangle
= \int_{D} \rho u_C^* \langle \alpha_{u(C)}, u_*\mu_C\rangle
= \int_{u(D)} \rho\circ u_D^{-1} \langle \alpha_{u(C)}, u_*\mu_C\rangle. 
\]
La contraction d'une $(n-1,n)$-forme sur~$\R^n$
avec un $n$-vecteur-volume est une opération ponctuelle;
comme $\alpha$ est supposée lisse, 
on peut écrire 
\[ \langle \alpha_{u(C)}, u_*\mu_C\rangle
= \langle \alpha, u_*\mu_C\rangle |_C. \]
Alors
\[ \int_{u(D)} \rho\circ u_D^{-1} \langle \alpha_{u(C)}, u_*\mu_C\rangle
= \int_{u(D)} \rho\circ u_D^{-1} \langle \alpha, u_*\mu_C\rangle  \]
et
\[ \int_{D} \sum_{C\in\mathscr C_D}
\langle (\rho u^*\alpha)_C, \mu_C \rangle
= \int_{u(D)} \rho\circ u_D^{-1} \langle \alpha, \sum_{C\in\mathscr C_D} u_*\mu_C\rangle. \]

Il suffit donc de prouver l'égalité
\[ \sum_{C\in\mathscr C_D} u_*\mu_C = 0  \]
pour toute cellule~$D$ de dimension~$(n-1)$ qui ne rencontre pas~$\partial(W)$.

Pour utiliser l'hypothèse que toute composante~$v$
de~$u$ est harmonique, on écrit l'égalité
$0=\langle \ddc v,\omega\rangle$ pour toute forme lisse~$\omega$
qui est à support compact dans~$\mathring W$.
Par la formule de Green (proposition~\ref{prop.stokes}), on a
\[
 0 = \int_X v \ddc\omega  
= \int_X  \ddc v \wedge\omega + \int_X^\partial
\left( v \dc\omega - \dc v \wedge\omega\right). \]
Le premier terme est nul car $v$ est~\pl,
de même que le second car le support de l'intégrale de bord
est contenu dans le bord de~$W$
(théorème~\ref{theo.support-mesure-bord},
$\int^\partial_W v\dc\omega$ est l'intégrale de~$v$
contre la mesure de bord associée à~$\dc\omega$).
Ainsi, on a 
\[ \int_X^\partial \dc v \wedge \omega = 0. \]

Fixons une $(n-1)$-cellule~$D\in\mathscr C$ dont l'intérieur ne rencontre
pas~$\partial(W)$. Soit $x\in\mathring D$.
Comme $\Int(W)\subset\Int(X)$, il existe 
une carte tropicale~$F$ sur un voisinage ouvert~$U$ de~$x$, à valeurs
un espace~$\R^N$,
telle que $\mathring D\cap U \subset \Sigma_F^{(n)}$
(proposition~\ref{prop.decent-extensible}).
En particulier, $F$ est une immersion par morceaux sur $\mathring D\cap U$.
Si $p\colon \R^N\to \R^{n-1}$ est une projection linéaire générique,
l'application affine~$f=p\circ F$ 
est une immersion par morceaux d'un voisinage de~$x$ dans~$D$
vers~$\R^{n-1}$.

Dans la suite, on va se contenter de considérer des formes~$\omega$
du type $\theta f^*\beta$, où $\theta$ est une fonction lisse
dont le support ne rencontre aucune $n$-cellule de~$\Sigma$
autres que celles qui contiennent~$D$,
et $\beta$ est une $(n-1,n-1)$-forme lisse sur~$\R^{n-1}$.

Choisissons un $(n-1)$-vecteur $\mu_D$, une orientation~$\eps_D$.
Pour toute $n$-cellule de~$\mathscr C$ contenant~$D$,
on en déduit une orientation de~$C$ (par la règle de la normale sortante)
et un vecteur~$e_C$ dans~$C$ tel que le 
calibrage canonique de~$C$ soit égal à $[e_C\wedge \mu_D, \eps_C]$.
Fixons aussi une fonction affine~$\phi_C$ sur~$C$ qui est nulle sur~$D$
et telle que $\phi_C(e_C)=1$.


Soit $a_C=\partial_{e_C}(v|_C)$ la dérivée de~$v|_C$ dans la direction de~$e_C$;
alors $v_C=v-a_C \phi_C$ est invariant par~$e_C$ et donc provient de~$D$
par la projection de~$C$ sur~$D$ parallèlement à~$e_C$.
Sur $C$, on a donc $\dc v= \dc v_C+ a_C \dc\phi_C$.
Compte tenu de la construction de~$\omega=\theta f^*\beta$,
lorsqu'on calcule le produit extérieur $\dc v\wedge \omega$,
qu'on le contracte par le calibrage canonique~$\mu_C$
et qu'on le restreint à~$D$, il reste
\[  a_C \theta \langle  f^*\beta|_D, \mu_D\rangle. \]
On a donc
\[ \int_D \sum_{C\supset D} a_C \theta \langle f^*\beta, \mu_D\rangle = 0. \]
Comme $\theta$ et $\beta$ sont arbitraires, on en déduit
$ \sum_{C\supset D} a_C f_*\mu_D = 0$;
en simplifiant par le $(n-1)$-vecteur non nul $f_*\mu_D$, on a ainsi
prouvé l'égalité
\[ \sum_{C\supset D} \partial_{e_C} (v|_C) = 0. \]

Inversement, ces égalités, lorsque $D$ parcourt l'ensemble
des cellules de dimension~$n-1$ de~$\mathscr C$,
entraînent que $\int_X^\partial \dc v \wedge\omega = 0$,
et donc que $\langle [\ddc v],\omega\rangle=0$.

  

Avec les notations précédentes,
$\mu_C=\mu_D\wedge e_C$ et $u_*(\mu_C)=u_*(\mu_D)\wedge u_*(e_C)$.
Le premier facteur $u_*(\mu_D)$ est indépendant de~$C$,
le second est le vecteur dont les coordonnées
sont les $a_C \phi_C(e_C)= \partial_{e_C}(v|_C)$, 
lorsque $v$ parcourt les coordonnées de~$u$.
On a donc
\[ \sum_{D\supset C} u_*(\mu_C) = u_*(\mu_D) \wedge \sum_{C\supset D} u_*(e_C)
= 0 \]
puisque les composantes du vecteur $\sum_C u_*(e_C)$ sont données
par $\sum_C \partial_{e_C}(v|_C)=0$.
\end{proof}

\begin{theo}\label{theo.harmonique-bord-nul}
Soit $u=(u_1,\dots,u_d)$ une suite finie de fonctions paralinéaires harmoniques.
Soit $P$ un sous-espace paralinéaire de~$\R^d$ contenant~$u(X)$
et soit $\alpha$ une forme lisse sur~$P$.
On a $[u^*\alpha]^\partial = 0$.
\end{theo}
\begin{proof}
Il s'agit de démontrer que $\langle [u^*\alpha]^\partial,\omega\rangle=0$
pour toute forme lisse~$\omega$ à support compact sur~$X$.
Considérons une partition de l'unité finie $(\rho_j)$ 
formée de fonctions lisses telle que $1=\sum\rho_j$
au voisinage du support de~$\omega$ et telle que pour tout~$j$,
$\omega$ soit lisse sur un voisinage du support de~$\rho_j$.
Puisque l'on a
\[ \langle [u^*\alpha]^\partial,\omega\rangle
= \sum_j \langle [u^*\alpha]^\partial,\rho_j\omega\rangle, \]
il suffit de traiter le cas où la forme-test considérée
s'écrit $\rho\omega$ où $\rho$ est une fonction lisse à support compact
et~$\omega$ est une forme tropicale.
Soit $f\colon X\to E$ une carte tropicale,
soit $P$ un sous-espace paralinéaire de~$E$ contenant~$f(X)$,
et soit $\beta$ une forme lisse sur~$P$ telle que $\omega=f^*\beta$.
On a
\[ \langle [u^*\alpha]^\partial, \rho\omega\rangle
= \int_{X}^\partial u^*\alpha \wedge \rho\omega
= \int_{X}^\partial \rho u^*\alpha \wedge f^*\beta
= \int_{X}^\partial \rho (u,f)^*\gamma \]
où $\gamma = p_1^*\alpha\wedge p_2^*\beta$ est une forme sur~$\R^n\times E$.
Comme $(u,f)$ est harmonique, la proposition précédente
entraîne la nullité recherchée.
\end{proof}

\begin{coro}[Condition d'harmonie] \label{coro.condition-harmonie}
Soit $X$ un espace tropical de niveau~$n$,
soit $u\colon X\to \R^d$ une application, harmonique.
Soit $\mathscr D$ une décomposition cellulaire de $\Pi_u$
adaptée au calibrage canonique~$\mu_{\Pi_u}$ et à~$\partial(X)$.
Toute $(n-1)$-cellule de~$\mathscr D$ qui n'est pas contenue
dans $u(\partial (X)\cap\Sigma_u))$ est harmonieuse.
\end{coro}
\begin{proof}
Soit $F$ une~$(n-1)$-cellule de~$\mathscr D$
qui n'est pas contenue dans $u(\partial (X)\cap\Sigma_u)$.
La discordance du calibrage de~$X$ le long de~$F$ relativement à~$u$
est un $(0,n)$-vecteur~$\delta_F$ de~$\R^d$; 
supposons par l'absurde qu'il n'est pas nul
et choisissons un multi-indice~$J$ de longueur~$n$ tel que $\langle \dc x_J,\delta_F\rangle\neq0$; choisissons aussi un multi-indice~$I$ de longueur~$n-1$
tel que la restriction à~$F$ de la $(n-1)$-forme $\di x_I$ ne soit pas nulle.

Soit $x$ un point de~$\mathring F$ qui n'appartient pas 
à $u(\partial(X))$.
Soit $\phi$ une fonction continue, non identiquement nulle,
et à support compact disjoint de~$u(\partial(X))$
et ne rencontrant pas d'autre face de dimension~$n-1$ que~$F$;
posons $\alpha=\phi\di x_I\wedge\dc x_J$ et $\omega=u^*\alpha$.
Par construction, $\langle\alpha, \mu_{\Pi_u}\rangle$
est une forme-volume non nulle.

D'après le théorème~\ref{theo.harmonique-bord-nul},
le support de la forme-volume~$\tilde\omega$ 
est contenu dans~$\partial(X)$.
Il est aussi disjoint
de~$\partial(X)$, par construction. On a donc $\tilde\omega=0$.
D'après la proposition~\ref{prop.formes-volume-pl-f*},
on a donc 
$\langle\alpha,\mu_{\Pi_u}\rangle =u_*(\tilde \omega)=0$,
ce qui contredit la construction de~$\alpha$.
\end{proof}

\subsection{}
On dit qu'une G-forme~$\omega$ sur~$X$ est \emph{harmoniquement tropicalisable}
s'il existe une application paralinéaire harmonique $u\colon X\to\R^d$
et une forme lisse~$\alpha$ au voisinage d'un polyèdre contenant~$u(X)$
telle que $\omega=u^*\alpha$. Elles forment une sous-algèbre
différentielle graduée de~$\mathscr A_{X^G}(X)$, invariante par~$\mathrm J$.
\index{G-forme!harmoniquement tropicalisable}

On définit alors le sous-faisceau $\mathscr A^{p,q}_{X,\harm}$
du faisceau $\pi_*\mathscr A^{p,q}_{X^\groth}$
(déduit du G-faisceau $\mathscr A^{p,q}_{X_G}$
par restriction à la catégorie des ouverts de~$X$) constitué
des formes qui sont localement harmoniquement tropicalisables.
C'est un sous-faisceau différentiel gradué de~$\pi_*\mathscr A_{X^G}$,
invariant par~$\mathrm J$.

\begin{coro}\label{coro.harm-diff}
L'inclusion $\mathscr A_{X,\harm} \to \mathscr D_{X}$
des G-formes localement harmoniquement tropicalisables 
dans les courants commute aux opérateurs différentiels~$\di$ et~$\dc$.
\end{coro}
\begin{proof}
Soit $\omega$ une G-forme localement harmoniquement tropicalisable.
Il faut démontrer les égalités $\di {[\omega]}=[\di\omega]$
et $\dc{[\omega]}=[\dc\omega]$. Via~$\mathrm J$, il suffit de traiter
la première relation.
L'assertion est locale sur~$X$, ce qui ramène au cas où $\omega$
est harmoniquement tropicalisable. Alors,
$\di{[\omega]}-[\di\omega]=[\omega]^\partial$, qui est nul
d'après le théorème~\ref{theo.harmonique-bord-nul}.
\end{proof}

\chapter{Théorie de Bedford--Taylor}

\def\Harm{\operatorname{Harm}\nolimits}
\def\PSH{\operatorname{Psh}\nolimits}
\def\LPSH{\operatorname{L^1Psh}\nolimits}
\section{Fonctions plurisousharmoniques}

Soit $X$ un espace G-tropical affiné.

\begin{defi}\label{defi.psh}
On note $\PSH_X$ le plus petit sous-faisceau du faisceau des
fonctions sur~$X$, à valeurs dans~$\R\cup\{-\infty\}$, 
tel que pour tout ouvert~$U$ de~$X$,
l'ensemble $\PSH_X(U)$ vérifie les propriétés:
\begin{enumerate}\def\theenumi{\roman{enumi}}\def\labelenumi{\theenumi)}
\item Il contient les fonctions lisses~$u$ sur~$U$ telles que $\ddc u\geq 0$;
\item Il est stable par limite décroissante de suites.
\end{enumerate}
Ses sections sont appelées des \emph{fonctions plurisousharmoniques (psh)}.
\index{fonction plurisousharmonique}\index{psh}
\end{defi}

\begin{lemm}\label{lemm.psh-cone}
Pour tout ouvert~$U$ de~$X$, l'ensemble des fonctions psh sur~$U$
est un cône convexe formé de fonctions semi-continues supérieurement;
il contient les fonctions constantes.
\end{lemm}
La démonstration est identique à celle du lemme~\ref{lemm.convexe-cone}.
\begin{proof}
Les fonctions constantes sont lisses et psh.

Les fonctions semi-continues supérieurement~$u$ telles 
que $\lambda u$ soit psh, 
pour tout nombre réel positif~$\lambda$, forment un sous-faisceau
de~$\PSH_X$ qui vérifie les conditions~(i) et (ii),
donc est égal à~$\PSH_X$. 
Cela prouve que pour tout ouvert~$U$ de~$X$,
l'ensemble $\PSH_X(U)$ est un sous-cône
de l'ensemble des fonctions semi-continues supérieurement sur~$U$.

Démontrons que si $u$ et $v$ sont des fonctions psh sur un ouvert~$U$
de~$X$, il en est de même de $u+v$.

Supposons d'abord que $u$ soit lisse.
Les fonctions~$f$ telles que $f+u$ soit psh forment un sous-faisceau
de~$\PSH_U$  qui contient les fonctions lisses psh et est stable
par limite décroissante; il contient donc toutes les fonctions psh sur~$U$.
En particulier, $u+v$ est psh lorsque $u$ est lisse.

Dans le cas général, les fonctions~$f$ telles que $f+v$ 
soit psh forme un sous-faisceau de~$\PSH_U$ qui contient les fonctions 
lisses psh d'après le premier cas traité et est stable par limite décroissante;
il contient donc toutes les fonctions psh sur~$U$;
par suite, $u+v$ est psh.
\end{proof}

\begin{rema}
Cette définition est inspirée par le théorème
de~\cite{FornaessNarasimhan-1980} en analyse complexe
qui affirme que dans ce contexte,
une fonction psh sur un espace de Stein
est limite d'une suite décroissante 
de  fonctions psh lisses.

Dans le présent travail, nous considérerons surtout 
des fonctions psh continues.

Par ailleurs, rappelons que si 
une suite décroissante de fonctions continues converge simplement 
vers une fonction continue,
elle converge même localement uniformément (théorème de Dini), 
puisqu'un espace tropical est localement compact.
\end{rema}

\begin{lemm}\label{lemm.PSH-ordinal}
Pour tout ordinal~$\alpha$, on définit par récurrence transfinie
un sous-faisceau $\PSH_{X,\alpha}$ du faisceau des
fonctions sur~$X$ par les conditions suivantes:
\begin{enumerate}
\item On définit $\PSH_{X,0}$ comme le faisceau des fonctions lisses~$u$
telles que $\ddc u\geq 0$;
\item Si $\alpha$ est un ordinal, on définit $\PSH_{X,\alpha+1}$
comme le sous-faisceau des fonctions sur~$X$ qui sont localement
limites décroissantes de fonctions de~$\PSH_{X,\alpha}$;
\item Si $\alpha$ est un ordinal limite, on définit $\PSH_{X,\alpha}$
comme le sous-faisceau des fonctions sur~$X$ qui appartiennent localement
à un $\PSH_{X,\beta}$ pour $\beta<\alpha$.
\end{enumerate}
La famille ordinale $(\PSH_{X,\alpha})$ est une famille croissante
de sous-faisceaux de~$\PSH_X$ qui vaut~$\PSH_X$ pour $\alpha$ assez grand.
\end{lemm}
\begin{proof}
Par récurrence transfinie, il résulte de la définition de~$\PSH_X$,
que $\PSH_{X,\alpha}\subset \PSH_X$ pour tout ordinal~$\alpha$.
La construction montre que cette famille de sous-faisceau est croissante.
Elle est nécessairement stationnaire, par un argument de cardinal.
Soit $\alpha$ un ordinal à partir duquel la famille est constante
et 
démontrons que $\PSH_X\subset \PSH_{X,\alpha}$.
Il suffit pour cela de démontrer que $\PSH_{X,\alpha}$ vérifie les
deux conditions de la définition de~$\PSH_X$.
Il est évident que $\PSH_{X,\alpha}$ contient les fonctions lisses~$u$
telles que $\ddc u\geq 0$, qui appartiennent au sous-faisceau~$\PSH_{X,0}$.
Soit maintenant~$U$ un ouvert de~$X$ et $(u_n)$ une suite
décroissante de fonctions de~$\PSH_{X,\alpha}(U)$ qui converge simplement
vers une fonction localement intégrable~$u$ et démontrons
que $u\in\PSH_{X,\alpha}(U)$. 
Par construction, on a $u\in\PSH_{X,\alpha+1}(U)$;
par définition de~$\alpha$, on a donc $u\in \PSH_{X,\alpha}(U)$.
\end{proof}

\begin{prop}\label{prop.uof-psh}
Soit $f\colon Y\to X$ un morphisme d'espaces G-tropicaux affinés.
Pour toute fonction psh~$u$ sur~$X$, la fonction $u\circ f$ sur~$Y$ est psh.
Plus précisément, pour tout ordinal~$\alpha$
tel que $u$ appartient à~$\PSH_{X,\alpha}$,
la fonction $u\circ f$ appartient à~$\PSH_{Y,\alpha}$.
\end{prop}
\begin{proof}
%
Compte tenu du lemme~\ref{lemm.PSH-ordinal}, il suffit
de démontrer la seconde assertion. On procède par récurrence sur~$\alpha$.
Le cas $\alpha=0$ résulte de ce que pour toute fonction psh lisse~$u$, 
la fonction~$u\circ f$ est psh lisse.
L'assertion découle alors de la construction par récurrence
des faisceaux~$\PSH_{X,\alpha}$ et $\PSH_{Y,\alpha}$.
\end{proof}

\begin{rema}
En géométrie complexe, on pourrait définir une hiérarchie ordinale 
analogue des faisceaux de fonctions psh, mais leur théorie
établit que cette hiérarchie est stationnaire à partir du cran~1.
Dans notre cadre, 
nous ne savons même pas si la hiérarchie ordinale des espaces
de fonctions psh est stationnaire.
Cependant, tous les exemples de fonctions~psh que nous connaissons 
appartiennent au premier cran de cette hiérarchie.
\end{rema}

\begin{defi}\label{defi.Cpsh}
On note $\CPSH_X$ le plus petit sous-faisceau du faisceau~$\mathscr C_X$
des fonctions continues sur~$X$ tel que 
pour tout ouvert~$U$ de~$X$, 
l'ensemble $\CPSH_X(U)$ vérifie les propriétés:
\begin{enumerate}\def\theenumi{\roman{enumi}}\def\labelenumi{\theenumi)}
\item Il contient les fonctions lisses~$u$ sur~$U$ telles que $\ddc u\geq 0$;
\item Il est stable par limite uniforme.
\end{enumerate}
\end{defi}
On note $\CPSH_X$ ce sous-faisceau;
ses sections sont appelées des \emph{fonctions continues plurisousharmoniques (cpsh)}.\index{fonction continue plurisousharmonique}\index{cpsh}

\begin{lemm}\label{lemm.CPSH-ordinal}
Pour tout ordinal~$\alpha$, on définit par récurrence transfinie
un sous-faisceau $\CPSH_{X,\alpha}$ par les conditions suivantes:
\begin{enumerate}
\item On définit $\CPSH_{X,0}$ comme le faisceau des fonctions lisses~$u$
telles que $\ddc u\geq 0$;
\item Si $\alpha$ est un ordinal, on définit $\CPSH_{X,\alpha+1}$
comme le sous-faisceau des fonctions sur~$X$ qui sont localement
limites uniforme de fonctions de~$\CPSH_{X,\alpha}$;
\item Si $\alpha$ est un ordinal limite, on définit $\CPSH_{X,\alpha}$
comme le sous-faisceau des fonctions sur~$X$ qui appartiennent localement
à un $\CPSH_{X,\beta}$ pour $\beta<\alpha$.
\end{enumerate}
La famille ordinale $(\CPSH_{X,\alpha})$ est une famille croissante
de sous-faisceaux de~$\CPSH_X$ qui vaut~$\CPSH_X$ pour $\alpha$ assez grand.
\end{lemm}
\begin{proof}
La démonstration est analogue à celle du lemme~\ref{lemm.PSH-ordinal}
en remplaçant partout « limite décroissante » par « limite uniforme ».
\end{proof}

\begin{prop}\label{prop.unif-psh}
Soit $(f_n)$ une suite de fonctions de~$X$ dans~$\R$
qui converge localement uniformément vers une fonction~$f$.
Si les~$f_n$ sont psh, alors $f$ est psh.
\end{prop}
\begin{proof}
Comme les fonctions constantes sont psh,
l'assertion découle du lemme suivant.
\end{proof}

\begin{lemm}\label{lemm.unif-decr}
Soit $X$ un ensemble et soit $(f_n)$ une suite de fonctions de~$X$ dans~$\R$
qui converge uniformément vers une fonction~$f$.
La suite $(f_n+2 \norm{f-f_n}_\infty)$ converge uniformément
vers~$f$ et admet une sous-suite qui est décroissante.
\end{lemm}
\begin{proof}
Puisque $\norm{f-f_n}_\infty$ tend vers~$0$, il est clair
que la suite $(f_n+2 \norm{f-f_n}_\infty)$ converge uniformément vers~$f$.
L'assertion est évidente s'il existe
une infinité d'entiers~$n$ tels que $\norm{f-f_n}_\infty=0$.
On peut donc supposer que $\norm{f-f_n}>0$ pour tout~$n$.
Soit $n$ un entier.
Puisque $(f_n)$ converge uniformément vers~$f$,
il existe $m>n$ tel que $\norm{f_m-f}_\infty < \norm{f-f_n}_\infty/3$.
Des inégalités 
$f_m\leq f+\norm{f_m-f}_\infty$ et $f-\norm{f_n-f}_\infty\leq f_n$,
on déduit que 
\[ f_m + 2 \norm{f_m-f}_\infty \leq f+3\norm{f_m-f}_\infty
\leq f_n + 3\norm{f_m-f}_\infty+\norm{f_n-f}_\infty
\leq f_n + 2 \norm{f_n-f}_\infty. \]
Le lemme s'en déduit.
\end{proof}

\begin{prop}\label{prop.CPSH=C-cap-PSH}
On a $\CPSH_X = \mathscr C_X\cap \PSH_X$ : les fonctions continues 
plurisousharmoniques sont les fonctions continues (à valeurs réelles)
qui sont  plurisousharmoniques.
Plus précisément, pour tout ordinal~$\alpha$,
on a $\CPSH_{X,\alpha}=\mathscr C_X\cap \PSH_{X,\alpha}$.
\end{prop}
\begin{proof}
Compte tenu des lemmes~\ref{lemm.PSH-ordinal}
et~\ref{lemm.CPSH-ordinal}, il suffit de vérifier 
la seconde assertion.
Prouvons cette égalité par récurrence transfinie. 
Les cas de l'ordinal~$0$ et d'un ordinal limite sont évidents.
Il suffit donc de prouver que pour tout ordinal~$\alpha$,
une fonction continue~$u$ définie sur un ouvert~$U$
est localement limite décroissante
de fonctions de~$\CPSH_{X,\alpha}$ si et seulement
si elle est localement limite uniforme de telles fonctions.

Soit $U$ un ouvert de~$X$
et soit $(u_n)$ une suite décroissante de fonctions de~$\CPSH_{X,\alpha}(U)$
qui converge simplement vers une fonction continue~$u$ sur~$U$.
Le théorème de Dini assure que la convergence est uniforme sur toute partie
compacte de~$U$.
Comme $U$ est localement compact, la convergence est
localement uniforme, si bien que la fonction~$u$ 
appartient à~$\CPSH_{X,\alpha+1}$.

Soit $U$ un ouvert de~$X$ et soit $u$ une fonction continue sur~$U$
qui est localement limite décroissante d'une suite
de fonctions de~$\CPSH_{X,\alpha}$.
Soit $(U_i)$ un recouvrement ouvert de~$U$
tel que $u|_{U_i}$ soit limite décroissante d'une suite
de fonctions de~$\CPSH_{X,\alpha}(U_i)$.
D'après ce qui précède, $u|_{U_i}$ appartient à~$\CPSH_{X,\alpha+1}(U_i)$.
Par la propriété de faisceau, $u$ appartient à~$\CPSH_{X,\alpha+1}(U)$.

Soit $U$ un ouvert de~$X$
et soit $(u_n)$ une suite de fonctions de~$\CPSH_{X,\alpha}(U)$
qui converge uniformément vers une fonction continue~$u$ sur~$U$.
D'après le lemme précédent, il existe une sous-suite de la suite
$(u_n+2\norm{u_n-u})$ qui converge uniformément, et en décroissant,
vers~$u$. Elles appartiennent à~$\CPSH_{X,\alpha}$,
donc à $\PSH_{X,\alpha}$, par l'hypothèse de récurrence,
si bien que $u$ appartient à~$\mathscr C_X\cap \PSH_{X,\alpha+1}$.

Soit $U$ un ouvert de~$X$ et soit $u$ une fonction continue sur~$U$
qui est localement limite uniforme d'une suite
de fonctions de~$\CPSH_{X,\alpha}$.
D'après ce qui précède, $u$ appartient localement à~$\mathscr C_X \cap \PSH_{X,\alpha+1}$,
donc est une section de ce faisceau sur l'ouvert~$U$.
\end{proof}

\begin{prop}\label{prop.psh-L1loc}
Le faisceau des fonctions psh et localement intégrables sur~$X$
est le plus petit sous-faisceau~$\LPSH_X$ du faisceau
des fonctions localement intégrables sur~$X$ tel que,
pour tout ouvert~$U$ de~$X$, 
l'ensemble $\LPSH_X(U)$ vérifie les propriétés:
\begin{enumerate}\def\theenumi{\roman{enumi}}\def\labelenumi{\theenumi)}
\item Il contient les fonctions lisses~$u$ sur~$U$ telles que $\ddc u\geq 0$;
\item Si $(u_n)$ est une suite décroissante d'éléments
de~$\LPSH_X(U)$ qui converge simplement vers une fonction
localement intégrable~$u$ sur~$U$, alors $u$ appartient à~$\LPSH_X(U)$.
\end{enumerate}
\end{prop}
\begin{proof}
On définit une suite $\LPSH_{X,\alpha}$ de sous-faisceaux
du faisceau des fonctions localement intégrables sur~$X$
de manière analogue à ce qui a été fait pour~$\PSH_{X,\alpha}$
et on démontre par récurrence transfinie que
$\LPSH_{X,\alpha}=\mathrm L^1_\loc\cap \PSH_{X,\alpha}$.
Seul le cas d'un ordinal successeur est non trivial;
supposant 
$\LPSH_{X,\alpha}=\mathrm L^1_\loc\cap \PSH_{X,\alpha}$,
démontrons que cette égalité persiste pour~$\alpha+1$.
L'inclusion
$\LPSH_{X,\alpha+1}\subset \mathrm L^1_\loc\cap \PSH_{X,\alpha+1}$
est évidente.

Inversement, soit $u$ une fonction de $\PSH_{X,\alpha+1}(U)$
qui est localement intégrable; il s'agit de démontrer
qu'elle appartient à~$\LPSH_{X,\alpha+1}(U)$.
Cela se vérifie localement sur~$U$
ce qui permet de supposer que $u$ est limite
décroissante d'une suite~$(u_n)$ de fonctions de~$\PSH_{X,\alpha}(U)$.
Pour tout~$n$, $u_n$ est semicontinue supérieurement, donc localement
majorée, et est minorée par la fonction localement intégrable~$u$,
si bien que $u_n$ est localement intégrable.
On a donc $u_n\in\mathrm L^1_\loc\PSH_{X,\alpha}=\LPSH_{X,\alpha}$ 
par l'hypothèse de récurrence. Ainsi, $u\in \LPSH_{X,\alpha+1}$.
\end{proof}

\begin{defi}
On dit qu'une fonction psh sur un ouvert de~$X$ est localement approchable
si c'est une section du faisceau~$\PSH_{X,1}$.
\end{defi}

\section{Convexité tropicale}

Soit $X$ un espace G-tropical affiné.

\begin{prop}\label{exem.psh-approchable}
Soit $u_1,\dots,u_n$ des fonctions psh sur~$X$.
Soit $\Omega$ un ouvert convexe de~$\R^n$ et 
soit $g\colon\Omega\to\R$ une fonction lisse, convexe, croissante par
rapport à chaque variable.
On suppose que $(u_1(x),\dots,u_n(x))\in\Omega$ pour tout $x\in X$.
Alors, la fonction  
$u\colon x\mapsto g(u_1(x),\dots,u_n(x))$ sur~$X$ est psh.

Plus précisément, si $\alpha$ est un ordinal
tel que pour tout~$i$, $u_i$ appartient~$\PSH_{X,\alpha}$,
la fonction~$u$ appartient à~$\PSH_{X,\alpha}$.
\end{prop}
\begin{proof}
On raisonne par récurrence sur $\alpha$.

Supposons d'abord $\alpha=0$,
c'est-à-dire que toutes les~$u_j$ sont lisses.
Alors $u$ est lisse et l'on a 
\[ \ddc u = \sum_{i,j} \frac{\partial^2 g}{\partial y_i\partial y_j}
 (u_1,\dots,u_n) \di u_i\wedge \dc u_j
+ \sum_{j} \frac{\partial g}{\partial y_j}\circ u\cdot \ddc u_j. \]
Puisque~$g$ est convexe, le premier terme
est une forme positive, en vertu de l'exemple~\ref{exem.forme-positive}.
Puisque $g$ est croissante par rapport à chaque variable,
$({\partial g}/\partial y_j)\circ u$ est une fonction positive sur~$X$;
par hypothèse, les $(1,1)$-formes $\ddc u_j$ sont positives;
par suite, le second terme est également positif. Cela démontre
que $\ddc u$ est une $(1,1)$-forme lisse positive,
si bien que la fonction~$u$ est donc psh.
Elle appartient donc à~$\PSH_{X,0}$.

Supposons maintenant $\alpha>0$.

Supposons que $\alpha$ soit un ordinal limite.
Soit $x$ un point de~$X$ ; il existe un voisinage ouvert~$U$ de~$x$ 
et un ordinal~$\beta<\alpha$ tels que
$u_k|_U \in \PSH_{X,\beta}(U)$ pour tout~$k$.
L'hypothèse de récurrence entraîne que
$g(u_1,\dots,u_n)|_U=g(u_1|_U,\dots,u_n|_U)$ appartient à
$\PSH_{X,\beta}$, donc à~$\PSH_{X,\alpha}$.
Il s'ensuit que $u$ appartient à~$\PSH_{X,\alpha}$.

Supposons maintenant que~$\alpha$ soit le successeur d'un ordinal~$\beta$.
Soit $x$ un point de~$X$ ; 
il existe un voisinage ouvert~$U$ de~$x$ et,
pour tout $k\in\{1,\dots,n\}$, 
une suite décroissante $(v_{k,m})$ de fonctions de~$\PSH_{X,\beta}(U)$
qui converge vers~$u_k|_U$.
Commme $g$ est continue et croissante par rapport à chaque variable,
la suite $(g(v_{1,m}|_U,\dots,u_{n,m}|_U))_m$
converge en décroissant vers $g(u_1,\dots,u_n)|_U$.
Par l'hypothèse de récurrence, ces fonctions appartiennent à~$\PSH_{X,\beta}$.
Par suite, $g(u_1,\dots,u_{n})$ appartient à $\PSH_{X,\alpha}$.
\end{proof}

\begin{coro}
On reprend les hypothèses de la proposition mais on suppose seulement
que $g$~est la limite d'une suite décroissante~$(g_k)$
de fonctions lisses, convexes, croissantes en chaque variable.
Si $\alpha$ est un ordinal
tel que pour tout~$i$, $u_i$ appartient à~$\PSH_{X,\alpha}$,
la fonction~$u$ appartient à~$\PSH_{X,\alpha+1}$.
Si la suite $(g_k)$ converge uniformémement vers~$g$,
alors $u$ appartient à~$\PSH_{X,\sup(\alpha,1)}$.
\end{coro}
\begin{proof}
La suite $(g_k(u_1,\dots,u_n))$ est décroissante
et converge simplement vers $g(u_1,\dots,u_n)$.
D'après la proposition, ces fonctions sont dans~$\PSH_{X,\alpha}$.
Par suite, $g(u_1,\dots,u_n)$ est dans~$\PSH_{X,\alpha+1}$.

Faisons l'hypothèse supplémentaire que la suite~$(g_k)$
converge uniformément vers~$g$ et démontrons
que $g(u_1,\dots,u_n)$ est dans $\PSH_{X,\sup(1,\alpha)}$.
On raisonne encore par récurrence sur~$\alpha$.
Le cas $\alpha=0$ est traité par la première partie;
supposons alors $\alpha>0$.

Supposons que $\alpha$ soit un ordinal limite.
Soit $x$ un point de~$X$ ; il existe un voisinage ouvert~$U$ de~$x$ 
et un ordinal~$\beta<\alpha$ tels que
$u_k|_U \in \PSH_{X,\beta}(U)$ pour tout~$k$.
La première partie de la démonstration affirme que
$g(u_1,\dots,u_n)|_U=g(u_1|_U,\dots,u_n|_U)$ appartient à
$\PSH_{X,\beta+1}$, donc à~$\PSH_{X,\alpha}$ puisque $\beta+1<\alpha$.

Supposons maintenant que~$\alpha$ soit le successeur d'un ordinal~$\beta$.
Soit $x$ un point de~$X$ ; 
il existe un voisinage ouvert~$U$ de~$x$ et,
pour tout $k\in\{1,\dots,n\}$, 
une suite décroissante $(v_{k,m})$ de fonctions de~$\PSH_{X,\beta}(U)$
qui converge vers~$u_k|_U$.
Pour tout entier~$m$, la fonction 
$g_m(v_{1,m}|_U,\dots,v_{n,m}|_U)$
appartient à~$\PSH_{X,\beta}$, car $g_m$ est lisse. 
On a 
\[ g_{m+1}(v_{1,m+1}|_U,\dots,v_{n,m+1}|_U)
\leq g_m (v_{1,m+1}|_U,\dots,v_{n,m+1}|_U)
\leq g_m (v_{1,m1}|_U,\dots,v_{n,m}|_U), \]
si bien que la suite $(g_m (v_{1,m1}|_U,\dots,v_{n,m}|_U))$
est décroissante.
Comme $(g_m)$ converge uniformément vers~$g$
et comme pour tout~$k$, la suite~$(v_{k,m})$  converge simplement 
vers~$u_k|U$, cette suite converge simplement
vers $g(u_1|_U,\dots,u_n|_U)$.
Cela démontre
que $g(u_1,\dots,u_n)|_U$ appartient à~$\PSH_{X,\beta+1}$,
c'est-à-dire à $\PSH_{X,\alpha}$.
\end{proof}

\begin{coro}
Soit $\alpha$ un ordinal et 
soit $u_1,\dots,u_n$ des fonctions psh sur~$X$,
appartenant à~$\PSH_{X,\alpha}$.
La fonction $x\mapsto \max(u_1(x),\dots,u_n(x))$ 
appartient à~$\PSH_{X,\sup(\alpha,1)}$,
et est en particulier~psh.
\end{coro}
\begin{proof}
D'après le lemme~\ref{lemm.approx-max},
la fonction $t\mapsto \max(t_1,\dots,t_n)$ sur~$\R^n$
est limite uniforme de la famille
de fonctions $t\mapsto \log(e^{at_1}+\dots+e^{at_m})/a$,
lorsque $a\to+\infty$. 
Ces fonctions sont lisses, convexes et croissantes en chaque variable.
L'assertion découle alors du corollaire précédent.
\end{proof}

\begin{coro}
Soit $\Omega$ un ouvert convexe de~$\R^n$ 
et soit~$g\colon\Omega\to\R$ une fonction convexe, croissante par rapport
à chaque variable. Pour toute suite $(u_1,\dots,u_n)$
de fonctions continues psh sur~$X$, 
la fonction $g(u_1,\dots,u_n)$ est continue et psh.
Plus précisément, si $\alpha$ est un ordinal
tel que $u_1,\dots,u_n$ appartiennent à~$\CPSH_{X,\alpha}$,
on a $g(u_1,\dots,u_n)\in \CPSH_{X,\sup(1,\alpha)}$.
\end{coro}
\begin{proof}
Comme $g$ est convexe sur l'ouvert~$\Omega$, elle est continue;
puisque les~$u_i$ sont continues, la fonction $g(u_1,\dots,u_n)$ est continue.
D'après la proposition~\ref{prop.CPSH=C-cap-PSH},
il suffit de démontrer qu'elle appartient à~$\PSH_{X,\sup(1,\alpha)}$.
D'après cette même proposition, les fonctions~$u_k$
appartiennent à~$\PSH_{X,\alpha}$.

Soit $x\in X$ et soit $\Omega'$ un voisinage ouvert convexe
relativement compact de $(u_1(x),\dots,u_n(x))$ dans~$\Omega$.
Comme les~$u_i$ sont continues,
la partie~$U$ de~$X$ formée 
des $x\in X$ tels que $(u_1(x),\dots,u_n(x))\in\Omega'$
est ouverte. 
Par convolution, $g|_{\Omega'}$ est limite uniforme d'une
suite $(g_k)$ de fonctions lisses qui sont convexes et croissantes
en chaque variable.
D'après le corollaire précédent,
$g(u_1,\dots,u_n)|_U$ appartient à~$\PSH_{X,\sup(1,\alpha)}$.
Cela conclut la démonstration.
\end{proof}
\begin{prop}\label{prop.convexe-psh}
Soit $X$ un espace G-tropical affiné.
Soit $f\colon X\to E$ une carte tropicale, 
soit $P$ un sous-espace paralinéaire de~$E$ contenant $f(X)$ 
et soit $v$ une fonction convexe continue sur~$P$
(définition~\ref{defi.convexe-pl}).
Alors la fonction $f^*v$ sur~$X$ est continue et psh.
\end{prop}
\begin{proof}
Le sous-faisceau de~$\mathscr C_P$ formé des fonctions~$v$
telles que $f^*v$ soit psh contient les fonctions lisses telles
que $\ddc v\geq 0$ et est stable par limite uniforme.
Il contient donc le faisceau des fonctions convexes sur~$P$.
Cela démontre la proposition.
\end{proof}

\begin{defi}\label{defi.trop-convexe}
Soit $X$ un espace G-tropical affiné.
On dit qu'une fonction~$u$ sur~$X$
est \emph{tropicalement convexe}
si pour tout point~$x$ de~$X$, il existe
un voisinage~$U$ de~$x$, 
une carte tropicale $f\colon U\to E$ sur~$U$,
une partie paralinéaire~$P$ de~$T_\trop$ qui contient $f(U)$
et une fonction convexe convexe~$v$ sur~$P$ tel que $u|_U=f^*v$.
\index{fonction tropicalement convexe}
\end{defi}

Une fonction tropicalement convexe est psh.

\section{Fonctions psh et courants positifs}

Soit $X$ un espace tropical de niveau~$n$.
On suppose que $X$ est sans bord;
en particulier, $X$ est localement différentiellement séparé,
très riche et décent.

\begin{prop}\label{prop.psh-courant}
\begin{enumerate}\def\theenumi{\alph{enumi}}\def\labelenumi{\theenumi)}
Soit $u$ une fonction sur~$X$ à valeurs dans~$\R\cup\{-\infty\}$,
localement intégrable.

\item
Si $u$ est plurisousharmonique, le courant $\ddc[u]$ est positif.

\item
Si $u$ est lisse, elle est plurisousharmonique si et seulement si
le courant $\ddc[u]$ est positif.
\end{enumerate}
\end{prop}
\begin{proof}
Pour tout~$U$, notons $\mathscr P(U)$ l'ensemble des fonctions
localement intégrables~$u$ sur~$U$ telles que le courant~$\ddc [u]$ sur~$U$
soit positif. Il est immédiat que $\mathscr P$ est un sous-faisceau
du faisceau des fonctions localement intégrables (à valeurs dans~$\R\cup\{-\infty\}$)
qui contient les fonctions lisses psh.
Vérifions qu'il est également stable par limite décroissante.
Soit $(u_n)$ une suite décroissante de fonctions localement intégrables
sur un ouvert~$U$ de~$X$ qui converge simplement 
vers une fonction~$u$ localement intégrable sur~$U$.
Soit $\omega$ une $(n,n)$-forme lisse et à support compact sur~$U$.
Si $\omega$ est positive, le théorème de convergence 
monotone entraîne que la suite $(\int u_n\omega)$ converge
vers $\int u\omega$. Comme $\omega$ est différence de deux formes
positives à support compact, cette convergence est encore vérifiée.
Cela démontre que la suite de courants $([u_n])$ 
converge vers le courant~$[u]$.
Si de plus $\ddc[u_n]\geq 0$ pour tout~$n$,
alors $\ddc[u]\geq 0$ par passage à la limite.
Par conséquent, $\mathscr P$ contient~$\CPSH_X$.
Cela démontre l'assertion~\emph a).

Supposons maintenant que $u$ est lisse.
Si $u$ est plurisousharmonique, alors $\ddc[u]$ est positif d'après~\emph a).
Supposons inversement que le courant~$\ddc[u]$ soit positif.
Comme c'est 
le courant associé à la forme lisse~$\ddc u$,
cette forme est positive, 
d'après le corollaire~\ref{coro.positivite-courant-forme}.
La fonction~$u$ appartient donc à~$\CPSH_X(U)$ par définition de ce cône.
\end{proof}
 
\begin{prop}
Soit $X$ l'espace tropical de niveau~$n$ associé à~$\R^n$.
Soit $U$ un ouvert convexe de~$\R^n$, soit $u$ une fonction sur~$U$
à valeurs dans~$\R\cup\{-\infty\}$ qui n'est pas identiquement~$-\infty$.
Les assertions suivantes sont équivalentes:
\begin{enumerate}\def\theenumi{\roman{enumi}}\def\labelenumi{(\theenumi)}
\item La fonction~$u$ appartient à~$\PSH_X$;
\item La fonction~$u$ est à valeurs finies et est convexe;
\item La fonction~$u$ est localement intégrable et vérifie~$\ddc [u]\geq0$;
\item La fonction~$u$ appartient à~$\CPSH_{X,1}$.
\end{enumerate}
En particulier, les faisceaux sur~$X$
des fonctions psh, des fonctions continues psh et
des fonctions localement convexes coïncident.
\end{prop}
\begin{proof}
\def\Conv{\operatorname{Conv}\nolimits}
Soit $U$ un ouvert de~$X$ et soit $u$ une fonction lisse sur~$U$
telle que $\ddc u\geq 0$. La hessienne de~$u$ est donc une forme
quadratique positive en tout point. Si $U$ est convexe,
on en déduit que $u$ est une fonction convexe sur~$U$.

(i)$\Rightarrow$(ii).
Soit $\Conv_X^*$ le sous-faisceau de~$\mathscr C_X$ constitué des fonctions
qui, localement, sont soit finies et localement convexes, 
soit identiquement $-\infty$. D'après ce qui précède,
il contient les fonctions lisses à $\ddc$ positif. 
Par ailleurs, soit $U$ un ouvert convexe de~$X$
et soit $(u_k)$ une suite de fonctions finies convexes sur~$U$
qui converge simplement vers une fonction~$u$ sur~$U$.
Démontrons que $u$ est soit finie et convexe, soit identiquement~$-\infty$.
On peut supposer que les~$u_k$ sont toutes finies et convexes.

Supposons qu'il existe $a\in U$ tel que $u(a)=-\infty$
et prouvons que $u$ est identiquement égale à~$-\infty$ sur~$U$
Soit $b\in U$; pour tout $t\in[0;1\mathclose[$,
on a $u_k((1-t)a+tb) \leq (1-t)u_k(a)+t u_k(b)$;
lorsqu'on passe à la limite, on obtient $u((1-t)a+tb)=-\infty$.
Ainsi, $u$ est $-\infty$ sur le segment semi-ouvert~$[a;b\mathclose[$.
Il en résulte que $u$ est identiquement~$-\infty$ sur~$U$.

Sinon, $u$ est finie, et en passant à la limite dans les inégalités
de convexité pour les~$u_k$, on voit que $u$ est convexe.

Il en résulte que le faisceau des fonctions~psh 
est contenu dans le faisceau~$\Conv_X^*$.

Démontrons~(iv) sous l'une des hypothèses (ii) ou~(iii).
Les fonctions~$u$ de~$\Conv_X^*$ qui ne sont pas identiquement~$-\infty$
sont finies et continues sur~$U$, donc localement intégrables.

L'assertion est locale; on peut donc supposer que $U$
est une boule ouverte~$B(a;r)$ et qu'il existe~$s>r$
tel que $\ddc [u]$ soit positif sur~$B(a,s)$.
Soit $\rho$ une fonction lisse, positive, à support compact,
d'intégrale~$1$; pour tout $\eps>0$, 
posons $\rho_\eps(x)=\eps^{-n}\rho(x/\eps)$.
Les convolées $u_\eps=u*\rho_\eps$ sont des fonctions lisses
qui convergent uniformément vers~$u$ sur $B(0;r)$ 
lorsque $\eps\to0$.
Démontrons que $u_\eps$ est convexe.
C'est le cas, par convolution, si $u$ est convexe.
Sinon, on a $\ddc[u]\geq 0$ par hypothèse,
si bien que $\ddc(u_\eps)=\ddc(u*\rho_\eps) =\ddc[u]*\rho_\eps\geq 0$,
et cela entraîne que $u_\eps$ est convexe.
Il en résulte que $u$ est limite uniforme sur~$B(a,r)$ 
de fonctions lisses et convexes. Elle appartient donc à~$\CPSH_{X,1}$.

L'implication~(iv)$\Rightarrow$(iii) résulte 
de la proposition~\ref{prop.psh-courant},
et l'implication~(iv)$\Rightarrow$(i) découle
de la proposition~\ref{prop.CPSH=C-cap-PSH}.
\end{proof}

\section{Inégalités de Chern--Levine--Nirenberg}

Nous nous inspirons ici du chapitre~2 de~\cite{demailly1985} 
pour démontrer des inégalités de type Chern--Levine--Nirenberg
(voir~\cite{ChernLevineNirenberg-1969}) dans notre contexte.

Soit $X$ un espace tropical de niveau~$n$, sans bord.

\begin{prop}
Soit $T$ un courant de bidegré~$(p,p)$, positif et fermé,
soit $q$ un entier tel que $0\leq q\leq p$ et
soit $\omega$ une forme lisse de type~$(p-q,p-q)$ sur~$X$.
Pour tout voisinage compact~$K$ du support de~$\omega$,
il existe un nombre réel~$c_{K,\omega}\geq0$ tel que 
pour toute famille $(u_0,\dots,u_q)$ de fonctions lisses sur~$X$,
telles que $u_1,\dots,u_q$ soient psh, on a 
\[ \Abs{\langle u_0 \ddc(u_1)\wedge\dots\wedge\ddc(u_q)\wedge T, \omega \rangle} \leq c_{K,\omega} \norm {u_0}_K  \norm{u_1}_K\dots \norm{u_q}_K. \]
\end{prop}
\begin{proof}
On prouve cette proposition par récurrence sur~$q$.
Considérons des formes lisses $\alpha_0$ et $\beta_0$,
positives et à support contenu dans~$\mathring K$
telles que $\omega=\alpha_0-\beta_0$
(proposition~\ref{prop.trop-diff-positives}).
Posons $\omega_0=\alpha_0+\beta_0$;
c'est une forme lisse positive, de bidegré~$(p-q,p-q)$
et à support compact contenu dans l'intérieur de~$K$.
Soit enfin~$C$ un nombre réel de sorte que la conclusion de la proposition
soit vraie lorsque $\omega$ est remplacée par la forme~$\ddc\omega_0$
et $q$ est remplacé par~$q-1$.

Soit $u_1,\dots,u_p$ des fonctions lisses sur~$X$,
et soit $u_0$ une fonction lisse sur~$X$.
On commence par écrire
\begin{multline*}
 \langle u_0 \ddc u_1\wedge \dots\wedge\ddc u_q\wedge T,\omega\rangle  \\
=\langle u_0 \ddc u_1\wedge \dots\wedge\ddc u_q\wedge T,\alpha_0\rangle
- \langle u_0 \ddc u_1\wedge \dots\wedge\ddc u_q\wedge T, \beta_0\rangle
.\end{multline*}
Comme les formes $\ddc u_1,\dots,\ddc u_q$ sont fortement positives,
et comme le courant~$T$ est positif,
$\ddc u_1\wedge\dots\wedge\ddc u_q\wedge T$ est positif
(lemm~\ref{lemm.produit-courant-positif-forme}).
Comme $\alpha_0$ est positive,
le premier terme est majoré par
\[ \norm{u_0}_K \langle \ddc u_1\wedge \dots\wedge\ddc u_q\wedge T,\alpha_0\rangle. \]
Un raisonnement analogue avec l'autre terme, ainsi qu'avec $-u_0$,
fournit la majoration
\[ \Abs{\langle u_0 \ddc u_1\wedge \dots\wedge\ddc u_q\wedge T, \omega\rangle}
\leq \norm{u_0}_{K} \langle  \ddc u_1\wedge \dots\wedge\ddc u_q\wedge T,\omega_0\rangle.
\]

Comme $T$ est fermé, on a 
\[ \ddc u_1\wedge \ddc u_2\wedge\dots\wedge \ddc u_q \wedge T =
\ddc (u_1 \ddc u_2\wedge\dots\wedge \ddc u_q \wedge T ) ,\]
de sorte que
\[ \langle \ddc u_1\wedge \dots\wedge\ddc u_q\wedge T,\omega_0\rangle
= \langle u_1 \ddc u_2\wedge \dots\wedge\ddc u_q\wedge T,\ddc\omega_0\rangle.
\] 
Par l'hypothèse de récurrence, 
on a donc 
\[ \langle \ddc u_1\wedge \dots\wedge\ddc u_q\wedge T,\omega_0\rangle
\leq  c_{K,\ddc\omega_0}  \norm{u_1}_K\dots\norm{u_q}_K, \]
d'où la majoration voulue.
\end{proof}

\subsection{}
On suppose que $X$ est sans bord.
Soit $\rho$ une fonction lisse psh sur~$X$
tel que l'ouvert~$U$ de~$X$ défini par $\rho<0$ 
soit relativement compact.

On pose $\rho_-=\sup(-\rho,0)$; c'est une fonction continue
et positive sur~$X$, lisse et strictement positive sur~$U$, nulle hors de~$U$.
Pour tout nombre réel~$a>0$,
la fonction $\theta_a\colon t\mapsto \sup(-t,0)^a$ sur~$\R$ 
est continue ; si $a>1$, elle est de dérivable
est vérifie $\theta_a'=-a\theta_{a-1}$; 
elle est donc de classe~$\mathscr C^{\lfloor a\rfloor}$;
par suite, la fonction $\rho_-^{a}$ sur~$U$ est à coefficients
de classe~$\mathscr C^{\lfloor a\rfloor}$;
on a $\di\rho_-^a=-a\rho_-^{a-1}\di\rho$
et $\dc\rho_-^a=-a\rho_-^{a-1}\dc\rho$.

Soit $a$ un nombre réel tel que $a>1$.
Pour tout~$k\in\{0,\dots,n\}$,
on définit une forme de type~$(n-k,n-k)$ sur~$U$ par 
\[  \beta_k = \rho_-^{k+a} (\ddc \rho)^{n-k} + 
    (k+a) \rho_-^{k+a-1} \di\rho \wedge \dc\rho \wedge (\ddc \rho)^{n-k-1}. \]
(Lorsque $k=n$, on omet le second terme.)
La forme~$\beta_k$ est à coefficients 
de classe~$\mathscr C^{\lfloor a\rfloor-1}$,
et son support contenu dans~$\overline W$,
donc est compact, car $k+a-1>0$.
Comme $\rho$ est lisse et psh, la $(1,1)$-forme lisse~$\ddc \rho$ est positive ;
de plus, $\di\rho\wedge\dc\rho$ est une $(1,1)$-forme positive ; 
par suite, la forme~$\beta_k$ est positive.

On a 
\[ \dc \beta_k = -2(k+a)\rho_-^{k+a-1} \dc \rho \wedge (\ddc\rho)^{n-k} \]
et 
\[ \ddc \beta_k = 2(k+a) \left(
    -\rho_-^{k+a-1} (\ddc\rho)^{n-k+1} + 
        (k+a-1) \rho_-^{k+a-2} \di\rho\wedge\dc\rho \wedge (\ddc \rho)^{n-k}
    \right) .\]
Par suite, on a l'inégalité de formes 
\[  -2(k+a) \beta_{k-1} \leq \ddc \beta_k \leq 2(k+a) \beta_{k-1}.\]

\begin{prop}\label{prop.cln-2}
On suppose $a\geq 3$ et on 
pose $C(k,a) = (4e)^k a^{-1}(k+a) (2+a) \dots (k+1+a)$.
Soit $V$ une fonction psh négative lisse sur~$U$.
Pour tout entier~$k$
et toute suite $(u_1,\dots,u_k)$ de fonctions lisses psh sur~$X$,
on a 
\[ \int_U \beta_k \wedge \abs V \ddc u_1\wedge \dots\wedge \ddc u_k
\leq C(k,a) \norm{u_1}_U \dots \norm{u_k}_U \int_U \abs V \beta_0 . \]
\end{prop}
\begin{proof}
On démontre l'inégalité par récurrence sur~$k$. 
Elle est évidente lorsque $k=0$, car $C(0,a)=1$.
Supposons maintenant $k\geq 1$.
Notons $J_k$ l'intégrale à majorer
et posons $\psi_k = \ddc u_1\wedge \dots\wedge \ddc u_k$.
Comme $\beta_k$ est ponctuellement nulle hors de~$U$, 
on a 
\[ J_k = \int_U (-V\beta_k )\wedge  \psi_{k-1}\wedge \ddc u_k
= \int_X (-V\beta_k )\wedge  \psi_{k-1}\wedge \ddc u_k
 = \int_X  (-V\beta_k) \wedge \ddc (u_k \psi_{k-1}). \]
Comme $X$ est sans bord, la formule de Green (proposition~\ref{prop.stokes})
entraîne alors
\[ J_k = \int_X \ddc (-V\beta_k) \wedge u_k \psi_{k-1} .
\]
Alors, 
\begin{align*}
J_k & = -\int_X \di (\dc V\wedge\beta_k +V\dc\beta_k) \wedge u_k \psi_{k-1} \\
& = - \int_X  (\ddc V \wedge \beta_k - \dc V \wedge\di\beta_k 
+ \di V \wedge\dc\beta_k + V \ddc\beta_k) \wedge u_k \psi_{k-1} \\
& = -\int_X  (\ddc V \wedge \beta_k 
+ 2\di V \wedge\dc\beta_k + V \ddc\beta_k) \wedge u_k \psi_{k-1}
\end{align*}
en utilisant la symétrie de~$V$, $\beta_k$ et $\psi_{k-1}$
(lemme~\ref{lemm.didc=-dcdi}).
Puisque
\[ \di (V \dc\beta_k \wedge u_k \psi_{k-1})
 = \di V \wedge \dc\beta_k \wedge u_k \psi_{k-1}
+ V u_k \ddc \beta_k \wedge \psi_{k-1}
- V \dc\beta_k \wedge \di u_k \wedge \psi_{k-1}, \]
la formule de Stokes (proposition~\ref{prop.stokes})
et le fait que $X$ soit sans bord
entraînent que
\[ \int_X \di V \wedge \dc\beta_k \wedge u_k \psi_{k-1}
 = \int_X V \dc \beta_k \wedge \di u_k \wedge\psi_{k-1}
 - \int_X V u_k \ddc \beta_k \wedge \psi_{k-1}. \]
si bien que 
\[ J_k = \int_X  u_k (V \ddc\beta_k -\ddc V \wedge \beta_k) \wedge\psi_{k-1}
 +2 V \di u_k \wedge \dc \beta_k  \wedge \psi_{k-1}. \]
En utilisant de nouveau que~$\beta_k$ est ponctuellement
nulle hors de~$U$, on a 
\[ J_k = \int_U  u_k (V \ddc\beta_k -\ddc V \wedge \beta_k) \wedge\psi_{k-1}
 +2 V \di u_k \wedge \dc \beta_k  \wedge \psi_{k-1}. \]

On a 
\[ 2V \di u_k \wedge \dc \beta_k 
 = - 4(k+a) V \rho_-^{k+a-1} \di u_k \wedge \dc \rho \wedge (\ddc \rho)^{n-k}.
\]
Pour deux fonctions~$u$, $v$ de classe~$\mathscr C^1$,
et pour $h$ strictement positive, 
la positivité de la forme $(\di u- h \di v)\wedge (\dc u - h \dc v)$ 
entraîne la majoration
\[ \di u \wedge \dc v - \di v \wedge \dc u
 \leq h^{-1} \di u \wedge \dc u + h \di v \wedge \dc v . \]
Lorsque $u_k$ est strictement positive, on en déduit la majoration
\begin{multline*}
 2 V \di u_k \wedge \dc \beta_k \wedge \psi_{k-1} \\
\leq 4(k+a)^2 (-V) u_k \rho_-^{k+a-2} \di\rho \wedge \dc\rho \wedge (\ddc\rho)^{n-k}
  \\
 + \rho_-^{k+a} (-V) u_k^{-1} \di u_k \wedge \dc u_k \wedge (\ddc \rho)^{n-k}.
\end{multline*}
Utilisant que $V$ est psh sur~$U$, donc $\ddc V\geq 0$, il vient alors
\begin{multline*}
J_k \leq \int_U u_k (-V) \left(-\ddc \beta_k 
    +4(k+a)^2 \rho_-^{k-2+a} \di\rho\wedge\dc\rho \wedge (\ddc \rho)^{n-k}\right) \wedge \psi_{k-1}  \\
{} + \int_U (-V) \rho_-^{k+a} (\ddc \rho)^{n-k} u_k^{-1} \di u_k \wedge \dc u_k
    \wedge\psi_{k-1} . \end{multline*}
Compte tenu du calcul de $\ddc\beta_k$,
la forme entre parenthèse dans la première intégrale est égale à
\[ 2(k+a) \rho_-^{k+a-1} (\ddc\rho)^{n-k+1} 
 + (k+a+1) \rho_-^{k+a-2} \di\rho\wedge\dc\rho \wedge (\ddc \rho)^{n-k}, \]
donc est majorée par $c\beta_{k-1}$, où $c=2(k+a)(k+a+1)/(k+a-1)$.
Ainsi, 
\[ J_k \leq c \norm{u_k}_U J_{k-1} + \int_U  (-V) \beta_k\wedge\psi_{k-1}
 \wedge u_k^{-1} \di u_k\wedge \dc u_k, \]
inégalité valable sous l'hypothèse que $u_k$
est strictement positive. 

Soit $B$ un nombre réel ; posons aussi $M=\norm{u_k}_U$.
nous allons appliquer l'inégalité précédente
en remplaçant $u_k$ par $v_k=\exp(Bu_k)$.  
On a  $\di v_k = \di (\exp(B u_k)) = B v_k \di u_k$, $\dc v_k = B v_k \dc u_k$
et
\[ \ddc v_k = \di( B \exp(B u_k) \dc u_k) 
 = B^2 v_k \di u_k \wedge \dc u_k + B v_k \ddc u_k 
  = v_k^{-1} \di v_k \wedge \dc v_k + B v_k \ddc u_k .\]
Comme $\ddc u_k$ est positive et $u_k\geq -M$,
on a donc
\[ \ddc v_k \geq B e^{-BM} \ddc u_k + v_k^{-1} \di v_k \wedge \dc v_k . \]
Par suite,
\begin{align*}
 Be^{-BM} J_k  & = \int_U \beta_k \wedge (-V) \psi_{k-1} \wedge (B e^{-BM}\ddc u_k) \\
& \leq \int_U \beta_k \wedge (-V) \psi_{k-1}\wedge \ddc v_k
 - \int_U \beta_k \wedge (-V) \psi_{k-1} \wedge v_k^{-1} \di v_k \wedge \dc v_k \\
& \leq c \norm{v_k}_U  J_{k-1} \\
& \leq c e^{BM} J_{k-1},
\end{align*}
si bien que $ J_k \leq c B^{-1} e^{2BM} J_{k-1}$.
En choisissant $B=1/2M$, on en déduit 
\[ J_k \leq 2ec \norm{u_k}_U J_{k-1}
 = 4e \frac{(k+a) (k+a+1)}{k+a-1} \norm {u_k}_U J_{k-1}. \]
La proposition s'en déduit.
\end{proof}

\section{Théorie de Bedford--Taylor: produit par un courant
positif de bidegré~$(1,1)$}
\label{ss.produits-courants-positifs-new}

Soit $X$ un espace tropical de niveau~$n$, sans bord.

\begin{lemm}\label{lemm.unTn-uT}
Soit $(T_n)$ une suite de courants positifs et
soit $(u_n)$ une suite de fonctions continues.
On suppose que la suite~$(T_n)$ converge vers un courant~$T$
et que la suite~$(u_n)$ converge localement uniformément
vers une fonction~$u$.
On a $ u_n T_n \to u T$.
\end{lemm}
\begin{proof}
On peut raisonner localement 
et supposer que $(u_n)$ converge uniformément vers~$u$.
Soit $\eps>0$ et 
soit $v$ une fonction lisse telle que $\norm{u-v}\leq \eps$;
on écrit
\[ u_n T_n-uT=(u_n-v) T_n+v(T_n-T)+(v-u)T. \]
Soit $\omega$ une forme lisse à support compact, 
et soit $\alpha,\beta$ des formes lisses positives à support compact
telles que $\omega=\alpha-\beta$.
Comme $T_n$ est positif, on a 
\[ \abs{\langle (u_n-v) T_n, \omega \rangle}
\leq  \abs{\langle (u_n-v) T_n, \alpha \rangle}+ 
  \abs{\langle (u_n-v) T_n, \beta \rangle}  
\leq \norm{u_n-v} \langle T_n,\alpha+\beta \rangle . \]
Comme $v$ est lisse, le second terme
\[ \langle v(T_n-T),\omega\rangle 
= \langle T_n-T, v\omega\rangle   \]
tend vers~$0$ quand $n$ tend vers~$+\infty$.
Enfin, comme $T$ est positif,
\[ \abs{ \langle (v-u)T,\omega\rangle}\leq \norm{v-u} \langle T,\alpha+\beta\rangle. \]
Cela démontre que
\[ \limsup_n \abs{\langle u_nT_n-uT,\omega\rangle}
\leq 2 \norm{u-v} \langle T,\alpha+\beta\rangle
\leq 2\eps\langle T,\alpha+\beta\rangle.  \]
Comme ceci vaut pour tout~$\eps$, on a 
\[ \lim_n \langle u_n T_n,\omega\rangle = \langle uT,\omega\rangle. \]
Cela prouve que $u_nT_n\to uT$.
\end{proof}

\subsection{}
Soit $T$ un courant de bidegré~$(p,p)$ sur~$X$ qui est positif et fermé.
L'hypothèse que $T$ est fermé signifie 
$\di T=0$ et $\dc T=0$. Par définition d'un courant positif,
le courant~$T$ est symétrique ;  ainsi,
chacune de ces annulations entraîne l'autre.

\begin{lemm}\label{lemm.d'd''-wedge}
Soit $S$ un courant fermé symétrique de bidegré~$(p,p)$, 
soit $q$ un entier tel que $q\leq p-1$ et 
soit $\alpha$ 
une forme lisse symétrique de type~$(q,q)$.
On a  \[ \ddc(\alpha\wedge S)=(\ddc\alpha)\wedge S.\]
\end{lemm}
\begin{proof}
Soit $\beta$ une forme lisse de type $(p-q-1,p-q-1)$; 
démontrons que l'on a
\[ \langle \ddc\alpha\wedge S, \beta \rangle
= \langle \alpha \wedge S, \ddc\beta\rangle . \]
Comme les courants $\ddc\alpha\wedge S$
et $\alpha\wedge S$ sont symétriques, on peut supposer
que $\beta$ est symétrique.
Par définition, on a 
\[ \langle \ddc\alpha\wedge S, \beta \rangle
- \langle \alpha \wedge S, \ddc\beta\rangle 
= \langle S, \ddc\alpha\wedge\beta- \alpha\wedge\ddc\beta\rangle. \]
Or, on a 
\begin{multline*} 
\di( \alpha\dc\beta-\dc \alpha\wedge\beta)
 = \di \alpha \wedge\dc\beta + \alpha \ddc\beta
- \ddc \alpha \wedge\beta+\dc \alpha\wedge\di\beta .\end{multline*}
De plus, on démontre 
comme dans la preuve du lemme~\ref{lemm.green-affine},
la forme de type~$(p,p)$
\[ 
 \di \alpha \wedge\dc\beta +\dc\alpha\wedge\di\beta \]
est antisymétrique.
%
%
Sa valeur sur le courant symétrique~$S$ est donc  nulle.
Comme $S$ est fermé, cela entraîne l'égalité voulue.
\end{proof}

\subsection{}
Soit $T$ un courant de bidegré~$(p,p)$ sur~$X$
qui est positif et fermé.
Soit $u$ une fonction psh (continue) sur~$X$. 
Comme est $T$ est positif,
on peut former le courant~$uT$ (exemple~\ref{coro.courant-positif})
et considérer le courant~$\ddc(uT)$.
Si $u$ est lisse, le lemme~\ref{lemm.d'd''-wedge} entraîne que l'on a 
\[ \ddc(uT)=\ddc(u \wedge T) = \ddc u \wedge T = T \wedge \ddc(u), \]
car $T$ est fermé.

En géométrie complexe, cette remarque est la base de la construction,
par Bedford--Taylor~\cite{bedford-t82}, de produits de courants.
Nous allons utiliser la même méthode 
dans ce contexte.
Pour poursuivre la construction, il faut garantir que le courant
$\ddc(uT)$ est encore positif.
C'est le cas si $T$ est positif et $u$ est psh lisse
et, plus généralement, si $u$ est psh en vertu du lemme suivant.

\begin{lemm}\label{lemm.bt-positif}
Soit $T$ un courant de bidegré~$(p,p)$ sur~$X$, positif  et fermé.
Pour toute fonction continue psh~$u$,
le courant $\ddc(uT)$ est positif et fermé.
\end{lemm}
\begin{proof}
Pour tout ouvert~$U$ de~$X$, 
soit $\mathscr P(U)$ l'ensemble des fonctions continues~$u$
sur~$U$ telles que $\ddc(uT|_U)$ soit un courant positif.
Le préfaisceau~$\mathscr P$ ainsi défini est un sous-faisceau de~$\mathscr C_X$.
Si $u$ est psh lisse, on $\ddc(uT)=\ddc(u)\wedge T$, 
donc $u$ appartient à~$\mathscr P$.
Si $(u_n)$ est une suite de fonctions de~$\mathscr P(U)$
qui converge uniformément vers une fonction~$u$,
la suite de courants $(u_nT|_U)$ converge vers $uT|_U$,
donc $\ddc(uT|_U)$ est limite de la suite de courants $(\ddc(u_nT|_U))$.
Ces courants étant positifs par hypothèse, le courant $\ddc(uT|_U)$ est positif et $u$ appartient à~$\mathscr P(U)$.
Cela entraîne que le faisceau~$\mathscr P$ contient le faisceau~$\CPSH_X$,
ce qu'il fallait démontrer.
\end{proof}

\begin{prop}\label{prop.bt-projection}
Soit $j\colon Y\to X$ un morphisme topologiquement propre d'espaces
tropicaux de niveaux~$m$ et~$n$ respectivement.
Soit $u$ une fonction~psh continue sur~$X$ et soit $S$ un courant
positif fermé sur~$Y$. Le courant $j_*S$ sur~$X$
est positif fermé et l'on a l'égalité
\[ \ddc(u)\wedge j_*S = j_*(\ddc (j^*u)\wedge S). \]
\end{prop}
\begin{proof}
L'ensemble des fonctions psh~$u$  qui vérifient cette relation
forment un sous-faisceau en $\R_+$-modules du faisceau
des fonctions continues sur~$X$.
Il est stable par limite uniforme.
Compte tenu de la définition du faisceau~$\CPSH_X$, il suffit
donc de vérifier qu'il contient les fonctions lisses~$u$
telles que $\ddc u\geq 0$.
Dans ce cas, pour toute forme lisse~$\omega$ à support
compact sur un ouvert~$U$ de~$X$, on a
\begin{align*} \langle \ddc(u j_*S),\omega\rangle
 & = \langle uj_*S, \ddc \omega\rangle 
  = \langle j_*S, u\ddc\omega \rangle  \\
& = \langle S, j^*u \ddc j^*\omega\rangle 
 = \langle j^* u S, \ddc j^*\omega \rangle
= \langle \ddc(j^* uS), j^*\omega\rangle \\
&= \langle j_* \ddc(j^*uS),\omega\rangle,\end{align*}
ainsi qu'il fallait démontrer.
\end{proof}

\section{Courants de Bedford--Taylor}
\label{ss.courant-BT}

\subsection{}
Soit $u$ et $v$ des fonctions psh telles que $h=u-v$ est lisse.
Alors on a $\ddc(uT)=\ddc(hT)+\ddc(vT) = \ddc(h)\wedge T + \ddc(vT)$,
car $h$ est lisse.
En particulier, si $\ddc h=0$, on en déduit que $\ddc (uT)=\ddc(vT)$.
Nous ne savons pas si cette égalité persiste en général,
c'est-à-dire sans supposer que $h$ est lisse. Cela nous conduit
à poser la définition suivante.

\begin{defi}\label{defi.courant-BT}
On dit qu'un courant positif fermé~$T$ est de \emph{Bedford--Taylor}\index{courant de Bedford--Taylor} si
pour tout ouvert~$U$ de~$X$ et
pour tout couple $(u,v)$ de fonctions psh sur~$U$ telles que
$\ddc(u)= \ddc(v)$, on a $\ddc(uT|_U)= \ddc(vT|_U)$.
\end{defi}

Si c'est le cas, il 
est légitime de noter $ \ddc(u)\wedge T$ le courant $\ddc(uT)$. 
On s'autorisera cependant cette notation même si $T$ 
n'est pas de Bedford--Taylor.
Cet abus de notation sera relativement sans danger
car tous les courants que nous construisons dans ce texte sont
de Bedford--Taylor.

\begin{prop}
Soit $(U_i)$ un recouvrement ouvert de~$X$ tel que
$T|_{U_i}$ est de Bedford--Taylor pour tout~$i$; alors $T$ est 
de Bedford--Taylor.
\end{prop}
\begin{proof}
En effet, soit $u,v$ des fonctions psh sur un ouvert~$V$
de~$X$ telles que $\ddc(u) = \ddc(v)$;
l'égalité de courants $\ddc(uT) = \ddc(vT)$ sur~$V$
résulte de sa validité sur chacun des~$U_i\cap V$.
On en déduit que les courants de Bedford--Taylor forment 
un sous-faisceau du faisceau des courants positifs fermés.
\end{proof}

\begin{exem}
Le courant~$\delta_X$ est de Bedford--Taylor.

En effet, $u\delta_X=[u]$ donc $\ddc(u\delta_X)=\ddc([u])=\ddc u$.
\end{exem}

\begin{prop}
Si $\alpha$ est une forme lisse positive fermée et~$T$ est de Bedford--Taylor,
alors il en est de même du courant $T\wedge\alpha$.
\end{prop}
\begin{proof}
Soit en effet des fonctions~$u,v$ psh telles que $\ddc(u) = \ddc(v)$; 
si $\alpha$ est une forme lisse à support compact, on a
\[ \ddc(u(T\wedge\alpha)) =\ddc( uT \wedge\alpha) =\ddc(uT)\wedge\alpha.\]
De même, $\ddc(v(T\wedge\alpha))=\ddc(vT)\wedge\alpha$.
Comme $T$ est un courant de Bedford--Taylor,
on a $\ddc(uT) = \ddc(vT)$, d'où
$\ddc(uT)\wedge\alpha = \ddc(vT)\wedge\alpha$.
Cela démontre l'égalité requise.
\end{proof}

\begin{prop}
Soit $(T_k)$ une suite de courants de Bedford--Taylor
qui converge faiblement vers un courant~$T$. Le courant~$T$ est
de Bedford--Taylor.
\end{prop}
\begin{proof}
Le courant~$T$ est positif fermé.
Si $u$ est une fonction psh, on a $\ddc(uT)=\lim \ddc(uT_k)$ 
(lemme~\ref{lemm.unTn-uT}). 
Il en résulte que $T$ est un courant de Bedford--Taylor.
\end{proof}

\begin{coro}
Soit $T$ un courant de Bedford--Taylor et soit $u$
une fonction continue positive. Le courant~$uT$ est de Bedford--Taylor.
\end{coro}
\begin{proof}
Cela se déduit de la proposition 
par approximation locale par une fonction lisse.
\end{proof}

\begin{prop}\label{prop.ddcuT-BT}
Soit $T$ un courant de Bedford--Taylor.
Pour toute fonction
continue psh~$u$, le courant $\ddc(uT)$ est de Bedford--Taylor.
\end{prop}
\begin{proof}
Soit $u$ une fonction psh sur~$X$.
Soit~$\mathscr P$ le sous-préfaisceau de~$\mathscr C_X$ tel que,
pour tout ouvert~$U$, $\mathscr P(U)$ soit l'ensemble
des fonctions continues~$f$ sur~$U$ telles que
$\ddc(fT)$ soit un courant de Bedford--Taylor; c'est un faisceau.
Si $f$ est lisse psh, alors $\ddc(fT)=\ddc(f) \wedge T$,
qui est de Bedford--Taylor,  puisque $T$ l'est,
de sorte que $\mathscr P$ contient les fonctions lisses psh.
Ce faisceau est également stable 
par limites uniformes (lemme~\ref{lemm.unTn-uT}).
Par suite, $\mathscr P$ contient le faisceau des fonctions psh.
En particulier, 
le courant~$\ddc(uT)$ est de Bedford--Taylor.
\end{proof}

\section{Théorie de Bedford--Taylor: itération}
\label{sec.bt-iter}
%

\subsection{}\label{ss.bt-def}
Soit $u_1,\dots,u_q$ des fonctions
psh sur~$X$ et soit~$T$ un courant positif fermé.
Compte tenu du lemme~\ref{lemm.bt-positif},
on définit par récurrence un courant positif fermé
\[ \ddc(u_1 \ddc(u_2 \ddc (\dots \ddc(u_qT)))). \]
Par construction, ce courant  est même $\ddc$-exact si $q\geq 1$.

Lorsque $T$ est un courant de Bedford--Taylor, 
on démontre par récurrence à l'aide de la proposition~\ref{prop.ddcuT-BT}
qu'il en est de même du courant
\[ \ddc(u_1 \ddc(u_2 \ddc (\dots \ddc(u_qT)))). \]
et que ce courant 
ne dépend que de la suite de courants $(\ddc(u_1),\dots,\ddc(u_q),T)$,
justifiant ainsi la notation.

Lorsque les $u_j$ sont lisses, ce courant est égal à 
\[ \ddc(u_1) \wedge \ddc(u_2) \dots \wedge \ddc( u_q) \wedge T. \]

Dans la suite de ce texte, nous nous autoriserons l'abus de notation
consistant à écrire ce courant
\[ \ddc(u_1) \wedge \ddc(u_2) \dots \wedge \ddc( u_q) \wedge T \]
même si $T$ n'est pas de Bedford--Taylor.

En prenant $T=\delta_X$, on a $\ddc(u)\wedge \delta_X = \ddc(u)$,
pour toute fonction~psh~$u$. 
Plus généralement, on a
\[ \ddc(u_1) \wedge \ddc(u_2) \dots \wedge \ddc( u_q) \wedge \delta_X 
= \ddc(u_1) \wedge \ddc(u_2) \dots \wedge \ddc(u_q).
\]

%

\begin{prop}\label{prop.BT-continu}
Soit $(T_k)$ une suite de courants positifs fermés qui converge
vers un courant positif fermé~$T$.
Soit $(u_0^k)$, $(u_1^k),\dots,(u_q^k)$ des suites de fonctions continues
sur~$X$ qui convergent localement
uniformément vers des fonctions $u_0,u_1,\dots,u_q$ sur~$X$.
On suppose que pour tout $j\in\{1,\dots,q\}$ et tout~$k$,
la fonction~$u_j^k$ est psh.
Alors, 
la suite de courants 
\[ u_0^k \ddc(u_1^k) \wedge \dots \wedge \ddc(u_q^q) \wedge T_k \]
converge vers le courant
\[ u_0 \ddc(u_1) \wedge \dots \wedge  \ddc(u_q)\wedge T. \]
\end{prop}
\begin{proof}
On raisonne par récurrence sur~$q$.
Pour $q=0$, l'assertion résulte du lemme~\ref{lemm.unTn-uT}.
Supposons $q\geq 1$ et l'assertion démontrée pour~$q-1$.
Pour tout~$k$, posons
$S_k=  \ddc(u_2^k \ddc(\dots\ddc(u_q^k T_k)))$
et $S= \ddc(u_2 \ddc(\dots\ddc(u_q T)))$;
ce sont des courants positifs fermés, par récurrence,
et l'on a $S_k\to S$.
D'après le lemme~\ref{lemm.unTn-uT}, on a $u_1^kS_k\to u_1 S$.
Par continuité de l'opérateur~$\ddc$ sur les courants,
on a donc $\ddc(u_1^kS_k)\to \ddc(u_1 S)$.
Comme pour tout~$k$, le courant $\ddc(u_1^kS_k)$ est positif fermé,
le courant $\ddc(u_1S)$ l'est également.
En appliquant de nouveau le lemme~\ref{lemm.unTn-uT},
on en conclut que 
$ u_0^k \ddc(u_1^k S_k)$ converge vers 
$ u_0 \ddc(u_1 S)$, ce qu'il fallait démontrer.
\end{proof}

Nous aurons besoin de la variante suivante.
\begin{prop}\label{prop.BT-continu'}
Soit $(T_k)$ une suite de courants positifs fermés qui converge
vers un courant positif fermé~$T$.
Soit $f\colon X\to \R^n$ une carte harmonique.
Soit $(v_0^k)$, $(v_1^k),\dots,(v_q^k)$ des suites de fonctions 
convexes continues sur~$\R^n$ qui convergent localement
uniformément vers des fonctions $v_0,v_1,\dots,v_q$;
pour tout~$i$ et tout~$k$, posons $u_i^k=f^*v_i^k$ et $u_i=f^*v_i$.
Alors, 
la suite de courants 
\[ u_0^k \ddc(u_1^k) \wedge \dots \wedge \ddc(u_q^q) \wedge T_k \]
converge vers le courant
\[ u_0 \ddc(u_1) \wedge \dots \wedge  \ddc(u_q)\wedge T. \]
\end{prop}
\begin{proof}
La démonstration est analogue à la précédente; on en reprend les notations.
Pour $q=0$, il s'agit du lemme~\ref{lemm.unTn-uT}
qui n'exige que la continuité des fonctions en jeu.
Supposons $q\geq 1$ et l'assertion démontrée par~$q-1$.
Posons $S_k=\ddc(u_2^k \ddc(\dots\ddc(u_q^k T_k)))$
et $S= \ddc(u_2 \ddc(\dots\ddc(u_q T)))$;
ce sont des courants positifs fermés, par récurrence,
et l'on a $S_k\to S$. D'après le lemme~\ref{lemm.unTn-uT},
on a $u_1^kS_k\to u_1 S$ d'où, par continuité de l'opérateur~$\ddc$
sur les courants, la convergence
$\ddc(u_1^kS_k)\to \ddc(u_1 S)$.
Pour tout~$k$, le courant~$\ddc(u_1^kS_k)$ est positif fermé 
en vertu du corollaire~\ref{coro.wedge-weak-psh},
si bien que le courant~$\ddc(u_1 S)$ est positif fermé.
En appliquant de nouveau le lemme~\ref{lemm.unTn-uT},
on en conclut que 
$ u_0^k \ddc(u_1^k S_k)$ converge vers 
$ u_0 \ddc(u_1 S)$, ce qu'il fallait démontrer.
\end{proof}

\begin{prop}
Soit $T$ un courant positif fermé.
Soit $u_1,\dots,u_m$ des fonctions psh et soit $\sigma$ une permutation
de~$\{1,\dots,m\}$; on a 
\[ \ddc(u_1)\wedge \ddc(u_2)\wedge\dots\wedge \ddc(u_m)\wedge T
= \ddc(u_{\sigma(1)})\wedge \ddc(u_{\sigma(2)})\wedge\dots\wedge \ddc(u_{\sigma(m)})\wedge T. \]
\end{prop}
\begin{proof}
Soit $v$ une fonction psh sur~$X$.
Soit $\mathscr P$ le sous-faisceau  de $\CPSH_X$
formé des fonctions psh~$u$
telles que $\ddc(u \ddc(vT))=\ddc(v \ddc(uT))$.
Si $u$ est lisse, on a 
\[ \ddc(u\ddc(vT))=\ddc(u)\wedge \ddc(vT) 
= \ddc( vT\wedge \ddc(u)) =\ddc(v \ddc(u T)), \]
de sorte que ce faisceau~$\mathscr P$ contient les fonctions lisses.
Il est stable par limites uniformes. 
Par suite, il contient le faisceau des fonctions psh.

Cela entraîne l'assertion voulue lorsque $m=2$.

Le cas où $\sigma$ est la transposition échangeant deux
entiers successifs 
en découle immédiatement.
Le cas général s'en déduit car $\sigma$ est composée de telles transpositions.
\end{proof}

\begin{prop}\label{prop.bt-projection-p}
Soit $j\colon Y\to X$ un morphisme topologiquement propres d'espaces
tropicaux de niveaux~$m$ et~$n$ respectivement.
Soit $u_1,\dots,u_p$ des fonctions~psh sur~$X$ et soit $S$ un courant
positif fermé sur~$Y$. Le courant $j_*S$ sur~$X$
est positif fermé et l'on a l'égalité
\[ \ddc(u_1)\wedge \ddc(u_2) \wedge  \dots \ddc(u_p) \wedge j_*S = j_*(\ddc (j^*u_1) \wedge \ddc(j^*u_2) \dots \wedge\ddc (j^*u_p) \wedge S). \]
\end{prop}
\begin{proof}
La démonstration se déduit de la proposition~\ref{prop.bt-projection}
par une récurrence immédiate sur~$p$.
\end{proof}

\section{Extension aux différences de fonctions psh}
\begin{lemm}\label{rema.lisse-diff-psh}
Soit $u$ une fonction lisse sur~$X$. 
Tout point~$x$ de~$X$ possède un voisinage sur lequel
$u$ s'écrit comme la différence de deux fonctions lisses~psh.

Plus précisément, soit $U$ un ouvert contenu dans un domaine
compact~$W$ sur lequel $u$ est tropicalisée par 
une carte tropicale~$f$, il existe des fonctions lisses psh~$u_1,u_2$
sur~$U$, $f$-tropicales, telles que $u|_U=u_1-u_2$.
\end{lemm}
\begin{proof}
Soit $f\colon W\to T$ un moment tropicalisant~$u$,
soit $P$ le polyèdre compact $f(W)$ 
et soit $v$ une fonction lisse au voisinage de~$P$ 
telle que $u|_W=v\circ f_\trop$.

D'après la proposition~\ref{prop.diff-positive},
il existe des fonctions lisses~$v_1$ et~$v_2$ au voisinage de~$P$
telles que $v=v_1-v_2$ et telles que les formes $\ddc v_1$
et~$\ddc v_2$ sont positives.
Les fonctions $u_1=v_1\circ f_\trop|_U$ et $v_2=u_2\circ f_\trop|_U$ sur~$U$
sont lisses, psh, et l'on a $u=u_1-u_2$.
\end{proof}

\begin{prop}\label{rema.G-lisse-diff-psh}
\begin{enumerate}
\item
Soit $u$ une fonction G-lisse sur~$X$.
Soit $W$ un domaine
analytique compact~$W$ sur lequel $u$ est tropicalisée par 
un moment~$f$, il existe des fonctions convexes~$v_1,v_2$
sur l'enveloppe convexe de~$f(W)$ telles que $u|_U=f^*v_1-f^*v_2$.

\item
Soit $u$ une fonction paralinéaire sur~$X$ et
soit $W$ un domaine analytique compact de~$X$ sur lequel $u$
est tropicalisée par un moment~$f$.
Il existe des fonctions paralinéaires convexes~$v_1,v_2$
sur l'enveloppe convexe de~$f(W)$ telles que $u|_U=f^*v_1-f^*v_2$.
\end{enumerate}
\end{prop}
\begin{proof}
D'après la proposition~\ref{prop.local-G-local},
il existe
un voisinage domanial compact~$V$ de~$x$,
une carte tropicale~$f\colon V\to E$ sur~$V$ et une fonction G-lisse~$v$
sur~$f(V)$ telle que $u|_V=f^*v$. Lorsque $u$ est paralinéaire, $v$ l'est.
D'après les propositions~\ref{prop.extension-G-lisse}
et~\ref{prop.G-lisse-d-convexe},
il existe alors des fonctions convexes~$v_1$ et~$v_2$
sur l'enveloppe convexe de~$f(V)$ telles que $v=v_1-v_2$ sur~$f(V)$,
que l'on peut choisir paralinéaires sur~$u$ l'est.
D'après la proposition~\ref{prop.convexe-psh},
les fonctions $u_1=f^*v_1$ et $u_2=f^*v_2$ sur~$V$ 
sont tropicalement convexes et sont paralinéaires si~$u$ l'est;
elles vérifient $u=u_1-u_2$.
\end{proof}

\begin{coro}
Soit $u$ une fonction G-lisse sur~$X$.
Tout point de~$X$ en lequel $X$ est décent
possède un voisinage sur lequel
$u$ s'écrit comme la différence de deux fonctions 
tropicalement convexes, que l'on peut choisir paralinéaires
si $u$ l'est.
\end{coro}

\subsection{} \label{ss.coro.produit-approchables}
\def\RPSH{\mathrm{CPsh}^{\pm}}
Soit $\RPSH_X$ le sous-faisceau en $\R$-espaces vectoriels
de~$\mathscr C_X$ engendré par le faisceau~$\CPSH_X$ des fonctions psh.
Ses sections seront appelées dpsh 
(« différence de fonctions plurisousharmoniques »).\index{fonction dpsh}
D'après le lemme~\ref{rema.lisse-diff-psh}, ce sous-faisceau 
contient les fonctions lisses.
De même, les fonctions localement différence
de deux fonctions tropicalement convexes\index{fonction tropicalement d-convexe}
seront appelées tropicalement d-convexes; de telles
fonctions sont dpsh.
Lorsque $X$ est décent, toute fonction G-lisse
est tropicalement d-convexe  (proposition~\ref{rema.G-lisse-diff-psh}).

\subsection{}\label{ss.bt-Rpsh}
Soit $T$ un courant positif fermé sur~$X$ de bidegré~$(q,q)$.
Il existe un unique morphisme $(p+1)$-linéaire 
sur  $\mathscr C^0_X\times (\RPSH_X)^p$,
\[ (u^0,u^1,\dots,u^p)\mapsto  u^0 \ddc u^1\wedge\dots\wedge\ddc u^p, \]
à valeurs dans le faisceau des 
courants de bidegré~$(p+q,p+q)$ sur~$X$
qui coïncide avec cette application 
lorsque $u^0,u^1,\dots,u^p$ sont psh
et $u^0$ est positive. 
Par construction, l'image de ce faisceau est constituée de différences
de deux courants positifs.
Il est symétrique en $(u^1,\dots,u^p)$.
Si $T$ est de Bedford--Taylor, 
il ne dépend que de~$u^0$ et des courants $\ddc u^1,\dots,\ddc u^p$.
Si, de plus, $u^0$ est dpsh, on a 
\[ \ddc (u^0 \ddc u^1 \wedge\dots\wedge \ddc u ^p)= 
 \ddc u^0 \wedge \ddc u^1 \wedge\dots\wedge \ddc u ^p.
\]

Soit $f^1,\dots,f^p$ des fonctions lisses telles que $\ddc (f^i)=0$.
On a 
\[ u^0 \ddc (u^1+f^1) \wedge\dots\wedge \ddc (u ^p+f^p) \wedge T= 
 u^0 \ddc (u^1) \wedge\dots\wedge \ddc (u ^p) \wedge T, \]
même si $T$ n'est pas de Bedford--Taylor. 
Autrement dit, ce morphisme passe au quotient en un morphisme
de source $\mathscr C^0_X \times (\RPSH_X/\Harm_X)^p$
où $\Harm_X$ est le faisceau sur~$X$
des fonctions lisses~$h$ telles que $\ddc h=0$.

%
 
\begin{rema}\label{rema.support-bt-squelette}
Supposons 
que $u^1,\dots,u^p$ sont tropicalement d-convexes
et soit~$\alpha$ une forme lisse de type $(n-p,n-p)$ sur~$X$.
Alors $u_0\mapsto \langle u_0\ddc u^1\wedge\dots\wedge \ddc u^p, \alpha\rangle$
est une mesure de Radon sur~$X$, car les~$u^j$ sont dpsh.

Démontrons que son support est localement contenu
dans une partie squelettique. (Il est donc 
contenu dans une partie squelettique de~$X$ si 
ce support est compact ou si $X$ est paracompact.)

Plus précisément, soit $U$ un ouvert de~$X$,
soit $f\colon U\to E$ une carte tropicale,
soit $P$ une partie paralinéaire de~$E$
contenant $f(U)$
et soit $a^1,\dots,a^p$ des fonctions d-convexes sur~$P$ 
telles que $u^j=f^*a^j$ pour $j\in\{1,\dots,p\}$,
soit enfin $\omega$ une forme de type~$(n-p,n-p)$
sur~$P$ telle que $\alpha=f^*\omega$.
Démontrons que la trace sur~$U$ du support de cette mesure de Radon
est contenu dans~$\Sigma_f^{(n)}$.

Lorsque les~$a^j$ sont lisses et convexes, la restriction à~$U$
de cette mesure est égale à la mesure associée
à la forme lisse positive $f^*(\ddc a^1\wedge \dots \wedge \ddc a^p\wedge \omega)$
et son support est contenu dans~$\Sigma_f^{(n)}$.
Cette propriété est locale sur~$X$, et est préservée
par limite uniforme, d'après la continuité du produit
de courants à la Bedford--Taylor.
Il en résulte qu'elle vaut dès que les~$a^p$ sont convexes.
Par multilinéarité, le résultat vaut également 
lorsque les~$a^j$ sont d-convexes.
\end{rema}

\begin{rema}
Lorsque les fonctions~$u^j$ comme ci-dessus sont par ailleurs G-$\mathscr C^2$,
l'expression $u^0 \ddc u^1\wedge\dots\wedge \ddc u^p$ peut être interprétée
comme une G-forme de type~$(p,p)$. Le lecteur prendra
garde que le courant 
défini au \S\ref{ss.coro.produit-approchables}
n'est pas en général celui associé  à cette G-forme.
Ce dernier courant ne jouera d'ailleurs pas de rôle dans la théorie,
ce qui rend inoffensif cet abus de notation.
\end{rema}

\begin{defi}\label{defi.MA}
Soit $X$ un espace tropical de niveau~$n$,  sans bord.
Si $u$ est une fonction dpsh sur~$X$, 
le courant $(\ddc u)^n$
est une mesure sur~$X$,
appelée \emph{mesure de Monge-Ampère} de~$u$, et notée $\MA(u)$.
Si $u$ est psh, cette mesure est positive.
\end{defi}

\begin{rema}\label{rema.courants-chgt-base-zero}
Soit $f\colon Y\to X$ un morphisme d'espaces tropicaux.
Soit $p$ un entier naturel tel que $p>\dim(X)$
et soit $u^1,\dots, u^p$ des fonctions dpsh sur~$X$.
Les fonctions $f^*u^1,\dots,f^*u^p$ sont dpsh
et l'on a 
\[ \ddc(f^*u^1)\wedge \dots \wedge \ddc(f^*u^p) = 0. \]

L'assertion est évidente si les $u^i$ sont lisses.
Dans ce cas, le courant
$ \ddc(f^*u^1)\wedge \dots \wedge \ddc(f^*u^p) = 0$
est associé à la forme lisse de même nom,
laquelle est 
l'image réciproque de la forme 
$\ddc(u^1)\wedge\dots \wedge \ddc(u^p)$ sur~$X$,
et celle-ci est nulle puisque $p>\dim(X)$.

Pour établir le cas général, on peut supposer que les $u^i$ sont psh.
Les suites $(u^1,\dots,u^p)$ de fonctions psh qui vérifient
la propriété requise forment un sous-faisceau
de~$(\CPSH_X)^p$. Par construction
du produit 
$\ddc(u^1)\wedge\dots \wedge \ddc(u^p)$,
il est stable par limite uniforme.
Il contient les suites de fonctions lisses d'après ce qui précède,
donc contient toutes les suites de fonctions psh.
\end{rema}

\section{Un exemple de calcul de mesure de Monge-Ampère}

\subsection{}
Soit $X$ un espace tropical de niveau~$n$, sans bord.
Soit $(f_1,\dots,f_m)$ une famille
finie de fonctions \emph{affines} sur~$X$ et soit $f\colon X\ra\R^m$
la carte tropicale correspondante.
Posons $g=\max(f_1,\dots,f_m)$.
C'est une fonction continue psh sur~$X$ (exemple~\ref{exem.psh-approchable}).
Le but de ce paragraphe est de fournir une description
tropicale de la mesure positive $\MA(g)$.

Pour cela, nous commençons par en étudier l'analogue paralinéaire.

\subsection{}
Soit $C$ un cône polyédral calibré de dimension~$\leq n$ dans~$\R^m$ 
et soit~$\xi$ son sommet. On suppose que le calibrage~$\mu$ de~$C$
est adapté à une décomposition cellulaire~$(C_j)_{j\in J}$
de~$C$ dont les cellules sont des cônes polyédraux de sommet~$\xi$;
notons~$\mu_j$ le vecteur-volume de~$C_j$.
Soit $V_j$ le sous-espace affine de~$\R^m$ engendré par~$C_j$.
 
Soit $I$ l'ensemble des $i\in\{1,\dots,m\}$ tels que $\xi_i=\max(\xi_1,\dots,\xi_m)$.
Pour tout nombre réel~$\eps>0$ et tout $x\in\R^m$, posons 
$g_\eps(x)=\eps \log (\sum_{i\in I} \exp(x_i)/\eps))$.
Cette fonction sur~$\R^m $ est convexe et lisse
(proposition~\ref{MA-delta_0/1}). Lorsque $\eps\to 0$,
elle converge uniformément vers $g_0\colon x\mapsto \max_{i\in I}x_i$.
En particulier, elle converge uniformément vers 
la fonction $g\colon x\mapsto \max (x_1,\dots,x_n)$
dans le voisinage  de~$\xi$ défini par les inégalités $x_j<\max(\xi_1,\dots,\xi_n)$ pour $j\not\in I$.

Soit $j\in J$.
Appliquons la proposition~\ref{MA-delta_0/1} aux cônes~$C_j-\xi$ (de sommet~$0$)
et aux formes linéaires $x\mapsto x_i$, pour $i\in I$. 
La mesure positive $\langle (\ddc g_1)^n,\mu_{j}\rangle$
sur~$C_j$ est absolument continue par rapport à la mesure
de Lebesgue et est de masse finie; notons~$\lambda_j$ cette masse.
C'est un nombre réel positif ou nul; il est strictement positif
si et seulement la famille $(x_i|_{V_j})_{i\in I}$ est très séparante
(définition \ref{defi.tres-separante}) ; un indice $i_0$ étant fixé, 
cela revient à demander que l'image de $V_j$ sous $(x_i-x_{i_0})_i$ 
soit de dimension $<n$. 
Pour toute fonction~$\psi$ sur~$C_j$,
continue et à support compact, on a
\[ \int_{C_j} \psi \langle (\ddc g_\eps)^n,\mu\rangle 
\to \lambda_j \psi(0) \]
lorsque $\eps\to 0$.

Posons $\lambda_C = \sum_{j\in J} \lambda_j$. 
C'est un nombre réel positif ou nul.
Il est nul si et seulement l'image sous $(x_i-x_{i_0})$
de $C$ est e dimension $<n$. 
Il ne dépend pas du choix de la décomposition cellulaire~$(C_j)$.

Observons que $C^{(n)}=\emptyset$ si et seulement si $\dim(C)<n$;
dans ce cas, on a $\lambda_C=0$.
Dans le cas contraire, $C^{(n)}$ est un cône de sommet~$\xi$
et l'on a $\lambda_C=\lambda_{C^{(n)}}$.

Soit $U$ un voisinage paralinéaire de~$\xi$ dans~$C$
et soit $\psi$ une fonction continue à support compact sur~$V$.
Appliquons la proposition~\ref{MA-delta_0/1} aux cônes~$C_j-\xi$ (de sommet~$0$)
et aux formes linéaires $x\mapsto x_i$, pour $i\in I$. 
On en déduit que  pour toute fonction~$\psi$ sur~$C$, continue et à support
compact, on a
\[ \int_{C} \psi \langle (\ddc g_\eps)^n,\mu\rangle 
\to \lambda_C \psi(0) \]
lorsque $\eps\to 0$.

\begin{rema}\label{rema.lambda-projection}
Comme ci-dessus,
considérons un cône polyédral calibré~$C\subset\R^m$ de sommet~$\xi$,
et soit $I$ l'ensembles des~$i$ tels que $\xi_i=\max(\xi_1,\dots,\xi_m)$.
Soit $p_I\colon \R^m\to\R^I$ la projection naturelle.
Alors $p_I(C)$ est un cône polyédral de~$\R^I$, de sommet~$p_I(\xi)$.
Posons $C'=C^{(p_I)}$ ; c'est le plus grand sous-espace paralinéaire
de~$C$ tel que  
l'application linéaire~$p_I$ induise une immersion par morceaux affine
surjective de~$C'$ sur~$p_I(C)^{(n)}$; c'est un cône.
Munissons alors $p_I(C)^{(n)}$
(ou $p_I(C)$, si l'on veut) du calibrage image de la restriction à~$C'$
du calibrage de~$C$. Il résulte de la construction
que l'on a $\lambda_C=\lambda_{p_I(C)}$.
\end{rema}

\begin{prop}\label{prop.ma-berk-local}
Soit $X$ un espace tropical de niveau~$n$, sans bord.
Soit $(f_1,\dots,f_m)$ une famille
finie non vide de fonctions affines sur~$X$ et soit $f\colon X\ra\R^m$
la carte tropicale correspondante. Soit $j\in \{1,\ldots,m\}$. 

Posons $g=\max(f_1,\dots,f_m)$.
Pour tout~$\eps>0$ et tout $x\in X$,
posons $g_\eps(x)= \eps \log(\sum_{i=1}^m \exp(f_i(x)/\eps))$.
Les fonctions~$g_\eps$ sont lisses, psh, tropicalisées par~$f$;
de plus, $g_\eps$ converge uniformément vers~$g$ quand $\eps$~tend vers~$0$.
En particulier, la fonction $g=\max({f_1},\dots,{f_m})$
sur~$X$ est psh. 

Pour tout~$x\in \Sigma_f^{(n)}$, notons
$\Pi_{f,x}$ le cône tropical de $f$ en 
$x$ (définition \ref{defi.cone-tropical})
et ~$\lambda_x$ le nombre réel
positif ou nul associé au cône calibré $\Pi_{f,x}$ de~$\R^m$. 
L'ensemble des $x\in \Sigma_f^{(n)}$ tels que $\lambda_x\neq 0$ est fermé
et discret dans~$X$ et l'on a
\[ \MA(g)=\sum_{x\in \Sigma_f^{(n)}} \lambda_x \delta_x.\]
En particulier, $\MA(g)$ est nul au voisinage de~$x$
si la dimension tropicale de~$(f_i-f_j)_i$ en $x$ est $<n$. 
\end{prop}
\begin{proof}
Soit $h$ la fonction~$(x_1,\dots,x_m)\mapsto\max(x_1,\dots,x_m)$ sur~$\R^m$;
pour tout $\eps>0$, soit $h_\eps$ la fonction
donnée par $(x_1,\dots,x_m)\mapsto \eps\log(\sum_{i=1}^m \exp(x_i/\eps))$.
D'après le lemme~\ref{lemm.approx-max}, la fonction~$h_\eps$ sur~$\R^m $
est convexe et
l'on a $g=f^*h=\lim_{\eps\ra0^+}f^*h_\eps$ (uniformément).
D'après la proposition~\ref{prop.convexe-psh}, la fonction~$g$ sur~$X$
est donc psh
et la mesure positive~$\MA(g)$ est limite
de mesures lisses $ \ddc (f^*h_\eps)^n$.

La partie squelettique $\Sigma_f^{(n)}$ de~$X$
est un support fort des formes $\ddc(f^*h_\eps)^n$;
elle contient donc le support de la mesure associée et,
par passage à la limite, le support de la mesure~$\MA(g)$.

Soit $x\in X$. 
D'après ce qui précède, la mesure~$\MA(g)$
est nulle au voisinage de~$x$ si $x\not\in\Sigma_f^{(n)}$.
On suppose dans la suite que $x\in\Sigma_f^{(n)}$. 

Soit $I$ l'ensemble des $i\in\{1,\dots,m\}$ tels que $g(x)={f_i(x)}$.
On a donc $g=\max_{i\in I}({f_i})$ au voisinage de~$x$.

Supposons que $x\not\in\Sigma_{p_I\circ f}^{(n)}$; dans ce cas,
le même argument que précédemment entraîne que~$\MA(g)$ est nulle 
au voisinage de~$x$. D'autre part, le cône
$p_I(\Pi_{f,x})$ est de dimension~$<n$, et donc $\lambda_x=0$
d'après la remarque~\ref{rema.lambda-projection}.
Sinon, on peut considérer le cône tropical de $p_I\circ f$
en~$x$, et son invariant~$\lambda$ coïncide avec celui de~$\Pi_{f,x}$,
comme on le voit en combinant la remarque~\ref{rema.trop-calibrage-projection}
et la remarque~\ref{rema.lambda-projection}.
On peut donc supposer que $I=\{1,\dots,m\}$.
On peut aussi supposer que $X$ est compact.

Soit $\mathscr C$ une décomposition cellulaire de $\Sigma^{(n)}_f$
en $f$-cellules commodes (lemme~\ref{lemm.exists-dec-cell-commode}).
Soit $V$ la réunion dans~$\Sigma^{(n)}_f$ des~$n$-cellules de~$\mathscr C$
qui contiennent~$x$; c'est un voisinage compact de~$x$ dans~$\Sigma^{(n)}_f$.
Posons $\xi=f(x)$.

Soit $\phi$ une fonction continue sur~$X$. 
dont le support est contenu dans~$V$. On a
\[ \int_X \phi \MA(g) = \lim_{\eps\to0^+} \int_X \phi \MA(f^*h_\eps)
 = \lim_{\eps\to0^+} \sum_{\substack{C\in\mathscr C \\ \dim(C)=n \\ x\in C}}
   \int_C  \phi \MA(f^*h_\eps). \]
Fixons une $n$-cellule $C\in\mathscr C$ contenant~$x$ et posons $D=f(C)$.
Par hypothèse, $D$ est une cellule forte de~$\R^m$
et $f$ induit un homéomorphisme paralinéaire de~$C$ sur~$D$;
soit $\psi\colon D\to \R$ l'unique fonction continue telle que
$\psi\circ f|_C=\phi|_C$.
Comme $C$ est une $f$-cellule commode,
l'image par $f$ du calibrage canonique~$\mu_C$ de~$C$
est un vecteur-volume constant~$\mu_D$ de~$D$.
On a  donc
\[ \int_C \phi \MA(f^*h_\eps) =  \int_D \psi \langle\ddc(h_\eps)^n,\mu_D\rangle.
\]
Lorsque $\eps\to 0$, le second membre converge vers $\lambda_D\psi(0)
=\lambda_D\phi(x)$,
où $\lambda_D$ est 
l'invariant~$\lambda$ du cône de sommet~$\xi$ engendré par~$D$,
calibré par le vecteur-volume~$\mu_D$.
Par suite,
\[ \int_X \phi\MA(f^*h_\eps) =\sum_{C\in\mathscr C_n} \int_C \phi\MA(f^*h_\eps)
\to \left(\sum_D \lambda_D \right) \phi(x). \]
Il reste à observer que $\lambda_x=\sum_D \lambda_D$,
ce qui découle du fait que le cône tropical $\Pi_{f,x}$
est le cône de sommet~$\xi$ engendré  par le polyèdre $f_\trop(V)$,
muni de l'image du calibrage canonique.

Nous avons ainsi démontré que $\MA(g)|_U=\lambda_x\delta_x$
pour tout voisinage ouvert~$U$ de~$x$ tel que $U\cap\Sigma_f^{(n)}\subset V$.

En particulier, le support de la mesure~$\MA(g)$ est discret;
il est fermé par construction.
La formule globale $\MA(g)=\sum_{x\in\Sigma_f^{(n)}} \lambda_x\delta_x$ en résulte.
\end{proof}

\section{Produits de courants squelettiques}

Soit $X$ un espace tropical de niveau~$n$, sans bord.

\begin{prop}\label{prop.di(uT)}
Soit $T$ un courant squelettique fermé  de bidimension~$(p,p)$ sur~$X$ et
soit $u$ une fonction G-lisse sur~$X$.

Soit $f$ une carte tropicale sur un ouvert~$U$ de~$X$, 
soit $r$ un entier naturel tel que $r\leq p$,
soit $\omega$ une forme lisse $f$-tropicale de type $(p-r,p)$ et
soit $\rho$ une forme lisse à support compact de type~$(r,0)$ sur~$U$; 
on a 
\[ \langle \di(uT),\rho\wedge\omega \rangle = \int_{P_f^T} \di(u)\wedge\rho\wedge\omega. \]

En particulier, on a $\di(uT)=\di u \wedge T=T\wedge \di u$.
\end{prop}
\begin{proof}
D'après l'exemple~\ref{exem.d'-paralineaire}, 
on a 
\[ \langle \di(uT),\rho\omega\rangle
= \int_{P_f^T} \di u \wedge \rho\wedge\omega - \int_{P_f^T}^\partial u \rho\wedge\omega. \]
Comme le courant~$T$ est fermé, ce même exemple affirme en outre
que
\[ 
- \int_{P_f^T}^\partial v\rho\wedge\omega 
    = \langle \di T, v\rho\wedge\omega \rangle = 0 \]
pour toute fonction lisse~$v$ sur~$U$.
Par densité des fonctions lisses dans les fonctions continues, 
on a donc $\int_{P_f^T}^\partial u\rho\wedge\omega=0$.
Ainsi, 
\[ \langle \di(uT),\rho\wedge\omega\rangle
= \int_{P_f^T} \di u \wedge \rho\wedge\omega
= \langle T\wedge \di u, \rho\wedge\omega\rangle, \]
en vertu de la proposition~\ref{prop.paralineaire-PfT}.

L'assertion est en particulier vérifiée si $r=0$, auquel cas
$\rho$ est une fonction lisse.
D'après la remarque~\ref{rema.formes-test}, cela démontre
l'égalité $\di(uT)=T\wedge\di u$.
\end{proof}

\begin{prop}\label{prop.ddcuT}
Soit $T$ un courant squelettique fermé de type~$(p,p)$ sur~$X$.
Soit $f\colon X\to \R^n$ une carte tropicale, 
soit $v$ une fonction G-lisse sur~$\R^n$ et soit $u=f^*v$.
Pour toute forme lisse $f$-tropicale~$\omega$ symétrique de type~$(p-1,p-1)$ 
et toute fonction lisse~$\rho$ sur~$X$
telles que le support de~$\rho\omega$ soit compact, on a 
\[ \langle \ddc(uT),\rho\omega\rangle 
  = \int_{P_f^T} \ddc u \wedge \rho\omega
 - \int_{P_f^T}^\partial \dc u \wedge \rho\omega. \]
\end{prop}
\begin{proof}
Soit $W$ un domaine compact de~$X$ qui est un voisinage du support
de~$\rho\omega$.
Soit $(V_i)$ une famille finie de domaines
compacts de~$X$ dont les intérieurs recouvrent~$W$
et tels que $\rho|_{\mathring V_i}$ soit tropicale, pour tout $i$.
Soit $V$ la réunion des~$V_i$.

Pour tout~$i$, soit $f_i$ une carte tropicale sur~$V_i$ tropicalisant~$\rho$
et soit $h_i$ la carte tropicale $(f_i,f|_{V_i})$.
D'après la proposition~\ref{prop.paralineaire-atlas}
et la remarque qui la suit, il existe un sous-espace
squelettique $p$-calibré~$Q$ contenant~$P_f^T\cap W$
tel que, 
pour toute forme lisse~$\theta$ dont le support est contenu
dans~$\mathring W$ de type~$(p-q,0)$
et toute forme lisse~$\eta$  de type~$(q,p)$
dont la restriction à~$\mathring V_i$ est $h_i$-tropicale, pour tout~$i$,
on ait
$ \langle T,\theta\wedge\eta\rangle=\int_Q \theta\wedge\eta$.

On écrit 
\[ \langle \ddc(uT),\rho\omega\rangle
=\langle uT,\ddc(\rho\omega)\rangle
= \langle T,u\ddc(\rho\omega)\rangle. \]
Sur~$V_i$, la forme $\ddc(\rho\omega)$  est $h_i$-tropicale;
la fonction~$u$ est G-lisse. Par construction de~$Q$, on a donc
\[ \langle T,u\ddc(\rho\omega)\rangle
= \int_Q u \ddc(\rho\omega). \]

Appliquons la formule de Green (proposition~\ref{prop.stokes-pl}),
ce qui est loisible puisque $u$ et~$\rho\omega$ sont symétriques;
on a:
\[ \int_Q u \ddc(\rho\omega) = \int_Q \ddc u\wedge\rho\omega
+ \int_Q^\partial u \dc(\rho\omega)- \dc u \wedge\rho\omega. \]

La G-forme $\ddc u \wedge\rho\omega$ sur~$X$ est
de type $(p,p)$; pour tout~$i$, sa restriction à~$V_i$
est $h_i$-tropicale;  son support est contenu dans~$K$.
En appliquant la proposition~\ref{prop.paralineaire-atlas},
on obtient donc
\[ \int_Q \ddc u \wedge\rho\omega
 = \langle T, \ddc u \wedge\rho\omega \rangle.
\]
La G-forme $\ddc u \wedge\omega$ sur~$X$ est $f$-tropicale
et de type~$(p,p)$.
Compte tenu de la proposition~\ref{prop.paralineaire-PfT}, on a ainsi
\[ \langle T, \ddc u \wedge\rho\omega \rangle 
= \langle T, \rho\ddc u \wedge\omega \rangle 
= \int_{P_f^T} \rho\ddc u \wedge\omega
= \int_{P_f^T} \ddc u \wedge\rho\omega
.\]

La fonction G-lisse~$v$ sur~$\R^n$ est limite localement
uniforme de fonctions lisses~$v_n$;
posons $u_n=f^*v_n$.
Pour tout~$n$, on peut écrire
\[
 \int_Q^\partial u_n \dc(\rho\omega)   = \int_Q \di(u_n\dc(\rho\omega))\\
 = \langle T, \di(u_n \dc(\rho\omega))\rangle, \]
en appliquant la proposition~\ref{prop.paralineaire-atlas}
à la forme $\di(u_n\dc(\rho\omega))$.
Puisque $T$ est fermé, on a ensuite
\[ \int_Q^{\partial} u_n \dc(\rho\omega)=
 - \langle \di T, u_n \dc(\rho\omega)\rangle= 0.
\]
En passant à la limite, on obtient 
\[ \int_Q^\partial u\dc(\rho\omega) = 0. \]

Démontrons maintenant que 
\[ \int_Q^\partial \dc u \wedge\rho\omega
= \int_{P_f^T}^\partial \dc u \wedge\rho\omega. \]
On commence par appliquer le théorème de Stokes
et la proposition~\ref{prop.paralineaire-atlas}:
\[ \int_Q^\partial \dc u \wedge\rho\omega
= \int_Q \di (\dc u \wedge\rho\omega)
= \langle T, \di (\dc u \wedge\rho\omega)\rangle . \]
On a 
\begin{align*}
 \di (\dc u \wedge\rho\omega) 
& = \ddc u \wedge\rho\omega
- \dc u \wedge \di (\rho\omega)  \\
& = \ddc u \wedge\rho\omega
- \dc u \wedge \di \rho\omega 
- \dc u \wedge \rho\di\omega .
\end{align*}
Les G-formes $\ddc u \wedge\omega$,
$\dc u \wedge\omega$ et $\dc u \wedge\di\omega$
sont $f$-tropicales de type $(p,p)$, $(p-1,p)$
et $(p,p)$ respectivement.
D'après la proposition~\ref{prop.paralineaire-atlas}, on a donc
\begin{align*}
\int_Q^\partial \dc u \wedge\rho\omega 
& =  \int_Q \left( \ddc u \wedge\rho\omega
- \dc u \wedge \di \rho\omega 
-  \dc u \wedge \rho\di\omega \right) \\
& = \langle T,  \ddc u \wedge\rho\omega
- \dc u \wedge \di \rho\omega 
-  \dc u \wedge \rho\di\omega \rangle \\
& =  \int_{P_f^T} \left( \ddc u \wedge\rho\omega
- \dc u \wedge \di \rho\omega 
-  \dc u \wedge \rho\di\omega \right) \\
& = \int_{P_f^T} \di(\dc u \wedge\rho\omega) \\
& = \int_{P_f^T}^\partial \dc u \wedge\rho\omega
. \end{align*}

Cela démontre la formule annoncée.
\end{proof}

\begin{theo}\label{theo.ddcuT-paralineaire}
Soit $X$ un espace tropical de niveau~$n$, sans bord.
Soit $T$ un courant squelettique \emph{fermé} 
de bidimension~$(p,p)$ sur~$X$ et soit $u$ une fonction paralinéaire sur~$X$.
Le courant $\ddc (uT)$ est squelettique.
\end{theo}
\begin{proof}
L'assertion étant locale, on peut supposer que $X$ est paracompact
et que $u$ est tropicale (proposition~\ref{prop.local-G-local}).

Il suffit de prouver que pour toute carte tropicale~$f$ sur
un ouvert~$U$ de~$X$,
il existe une partie squelettique $(p-1)$-calibrée~$Q_f^T$ de~$U$
telle que pour
toute forme~$\omega$ $f$-tropicale de type~$(p-1,p-1)$, symétrique
et toute fonction lisse à support compact~$\rho$ sur~$U$, 
on a $\langle\ddc(uT),\rho\omega\rangle=\int_{Q_f^T} \rho\omega$.
On peut supposer la carte~$f$ assez fine pour que $u$
soit $f$-tropicale. Pour simplifier les notations, on suppose $U=X$.

Le courant $\ddc(uT)$ est symétrique.

Comme $u$ est paralinéaire, la G-forme $\ddc u$ 
sur~$X$ est nulle, de sorte que, d'après la proposition~\ref{prop.ddcuT},
on a 
\[ \langle \ddc (uT),\rho\omega \rangle
 = - \int_{P_f^T}^\partial  \dc u \wedge\rho\omega \]
pour toute forme lisse $f$-tropicale~$\omega$
de type $(p-1,p-1)$ et toute fonction lisse~$\rho$
telles que $\rho\omega$ soit à support compact.

Soit $\mathscr C$ une décomposition cellulaire de~$P_f^T$ 
en $f$-cellules,
 assez fine pour que la restriction de~$u$
à chaque cellule de~$\mathscr C$ soit $f$-affine.
Pour tout entier~$d$,
notons~$\mathscr C_d$ l'ensemble des cellules de dimension~$d$ de~$\mathscr C$.

On a 
\[  \int_{P_f^T}^\partial  \dc u \wedge\rho\omega
= \sum_{C\in\mathscr C_p}  \int_C^\partial \dc u \wedge\rho\omega. \]

L'intégrale de bord~$\int_C^\partial \dc u\wedge\rho\omega$ 
est une somme d'intégrales
sur des cellules~$D\in\mathscr C_{p-1}$ de réunion~$\partial(C)$.
La formule précédente montre que la forme linéaire
$\rho\mapsto \langle \ddc(uT),\rho\omega\rangle$
sur l'espace des fonctions lisses à support compact
est une mesure supportée par la réunion de ces cellules,
absolument continue par rapport à leur mesure de Lebesgue.
Pour déterminer cette mesure,
il suffit donc de la calculer après restriction à l'intérieur
d'une telle cellule~$D\in\mathscr C_{p-1}$.
On peut ainsi supposer que le support de~$\rho$ ne rencontre
aucune autre cellule de dimension~$p-1$ que~$D$.

Choisissons une forme affine~$v_D$ sur~$\R^n$
telle que $u|_D= f|_D^*v_D$. 
Si l'on remplace~$u$ par $u-f^*v_D$, on ne change pas le courant à calculer, 
car $f^*v_D$ est lisse et $T$ est fermé, de sorte 
que $\ddc(f^*v_D T)=\ddc(f^*v_D)\wedge T=0$,
étant donné que $\ddc(f^*v_D)=0$. 
On peut donc supposer que $u|_D=0$.

Choisissons une orientation~$\eps_D$ de~$f(D)$
et un $(p-1)$-vecteur non nul~$\nu_D$ sur~$f(D)$, 
de sorte que $f^*[\nu_D,\eps_D]$ est un vecteur-volume non nul sur~$D$.

Soit $C\in\mathscr C_p$ une cellule telle que $D\subset C$.
Travaillons un instant en coordonnées: 
supposons donc que $D=\{0\}\times\R^{p-1}$ et $C=\R_-\times\R^{p-1}$ 
munis de leurs orientations canoniques~$\eps_D$ et~$\eps_C$. Munissons~$D$
du $(p-1)$-vecteur $\nu_D=e_2\wedge\dots\wedge e_p$ et~$C$
du $p$-vecteur $\nu_C=  e_1\wedge\dots\wedge e_p$.
Il existe un unique nombre réel~$\lambda_C$
tel que le vecteur-volume de~$C$ soit donné par $\mu_C=\lambda_C[\nu_C,\eps_C]$.
Il existe aussi un unique nombre réel~$a_C$
tel que l'on ait $u(x_1,\dots,x_p)=a_C x_1$ sur~$C$.
On peut uniquement décomposer la forme~$\omega|_C$ sous la forme
\[ \omega|_C = \omega_0+ \di x_1 \wedge\omega_1 + \dc x_1 \wedge\omega_2
+ \di x_1\wedge \dc x_1\wedge\omega_3, \]
où les $\omega_i$ sont des formes ne faisant intervenir que
les variables $x_2,\dots,x_p$, respectivement de type
$(p-1,p-1)$, $(p-2,p-1)$, $(p-1,p-2)$ et $(p-2,p-2)$.
Avec ces notations,  on a aussi $\omega|_D=\omega_0|_D$.
Puisque le support de~$\rho$ ne rencontre aucune face de~$\partial(C)$
autre que~$D$,
on a 
\begin{align*}\allowdisplaybreaks[1]
 \int^\partial_C \langle \dc u \wedge\rho \omega ,
\mu_C \rangle
& = \int_D a_C \lambda_C \langle \dc x_1 \wedge \rho \omega, [\nu_C,\eps_C] \rangle\\
& = \int_D a_C \lambda_C \langle \dc x_1 \wedge \rho \omega, [e_1\wedge \nu_D,\eps_C] \rangle\\
& = \int_D a_C \lambda_C \langle \dc x_1 \wedge \rho (\omega_0+\di x_1\wedge\omega_1), [e_1\wedge\nu_D,\eps_C] \rangle\\
& = a_C\lambda_C  \int_D \langle \rho (\omega_0+\di x_1\wedge\omega_1)|_D,[\nu_D,\eps_D] \rangle \\
& = a_C\lambda_C \int_D \langle \rho\omega_0|_D,\nu_D\rangle\\
& =a_C\lambda_C  \int_D \langle \rho\omega|_D,\nu_D\rangle.
\end{align*}

Soit $\mu_D$ 
le $(p-1)$-vecteur-volume sur~$D$ défini par
$\mu_D=(\sum_{C\supset D} a_C\lambda_C) f^*[\nu_D,\eps_D]$.
Comme le support de~$\rho$ ne rencontre
aucune face de~$\partial(C)$ autre que~$D$, on a  
\[ \langle \ddc (uT),\rho\omega\rangle
  =  \int_D  \langle \rho\omega,\mu_D\rangle. \]

Lorsque~$D$ varie,
ces vecteurs-volumes~$\mu_D$ munissent la réunion~$Q$ des cellules
de dimension~$p-1$ de~$\mathscr C$ d'un $(p-1)$-calibrage~$\mu_Q$
telle que pour toute forme lisse~$\omega$,
$f$-tropicale et de type~$(p-1,p-1)$
et toute fonction lisse~$\rho$ 
telles que le support de~$\rho\omega$ soit compact, on ait
\[ \langle \ddc(uT),\rho\omega\rangle = \int_Q \theta\omega. \]

Cela conclut la démonstration du théorème.
\end{proof}

\begin{rema}
Soit $T$ un courant squelettique \emph{fermé}.
Les arguments de la proposition~\ref{prop.ddcuT} et du théorème~\ref{theo.ddcuT-paralineaire}
entraînent un résultat plus général sur la structure du courant
$\ddc(uT)$ lorsque $u$ est seulement supposée G-lisse.
D'après la proposition~\ref{prop.ddcuT},
ce courant est la somme de deux termes.

Le premier terme est le courant $T\wedge \ddc u$,
produit du courant paralinéaire~$G$ par la $G$-forme $\ddc u$.

Le second terme, donné par $\rho\omega\mapsto -\int_{P_f^T}^\partial
\dc u\wedge\rho\omega$ lorsque $\omega$ et $u$ sont $f$-tropicales,
pour un moment~$f$, a été explicité dans la preuve
du théorème~\ref{theo.ddcuT-paralineaire} lorsque $u$ est paralinéaire.
Cependant, en reprenant les arguments de cette démonstration,
on constate qu'on y construit (si $X$ est paracompact
et $u$ $f$-tropicale) un espace paralinéaire~$Q_f^T$
muni d'un « calibrage G-lisse » tel que ce terme
soit égal à $\rho\omega\mapsto \int_{Q_f^T} \rho\omega$,
lorsque $\omega$ est est $f$-tropicale
et $\rho\omega$ à support compact.
Dans les calculs de \emph{loc. cit.}, il suffit
en effet de remplacer le nombre réel~$a_C$ par la dérivée
partielle $\partial u/\partial x_1$, qui est une fonction
lisse sur chaque cellule~$D$.
\end{rema}

\begin{coro}\label{coro.ddcuT-paralineaire}
Soit $X$ un espace tropical de niveau~$n$, sans bord, 
soit $T$ un courant squelettique fermé sur~$X$ 
et soit $u_1,\dots,u_m$ des fonctions paralinéaires sur~$X$.
Le courant $\ddc(u_1)\wedge\dots \wedge \ddc(u_m)\wedge T$ 
est paralinéaire.
\end{coro}
\begin{proof}
Compte tenu de la définition
\[ \ddc(u_1) \wedge\dots \wedge \ddc(u_m)\wedge T
 \ddc ( u_1 \wedge \ddc(u_2)\wedge \dots \wedge \ddc(u_m)\wedge T), \]
le corollaire se déduit du théorème~\ref{theo.ddcuT-paralineaire}
par récurrence sur~$m$.
\end{proof}

\section{Tropicalisations harmoniques et courants de Bedford--Taylor}

On démontre ici plusieurs conséquences du théorème~\ref{theo.harmonique-bord-nul}
pour les courants de Bedford--Taylor.

\begin{prop}\label{prop.di-T-wedge=}
Soit $X$ un espace tropical de niveau~$n$, sans bord,
soit $T$ un courant squelettique sur~$X$ qui est de Bedford--Taylor
et soit $\alpha$ une G-forme sur~$X$
qui est localement harmoniquement tropicalisable.
On a les égalités
\[ \di(T \wedge\alpha) =  T \wedge \di\alpha 
\quad\text{et}\quad
 \dc(T \wedge\alpha) = T \wedge \dc\alpha .\]
\end{prop}
\begin{proof}
Par partition de l'unité, il suffit de démontrer 
cette égalité de courants contre une forme du type $\rho\omega$ 
où $\rho$ est une fonction lisse dont
le support est contenu l'intérieur~$\mathring W$ d'un domaine
compact~$W$ de~$X$ 
telles qu'il existe une carte tropicale~$f$ sur~$W$ 
qui tropicalise~$\alpha$ et $\omega$;
on exige également que~$\alpha$ s'écrive~$u^*\beta$,
où $u\colon W\to\R^d$ est une carte G-tropicale 
dont la restriction à~$\mathring W$ est harmonique.

Notons $(p,q)$ le bidegré de~$\alpha$;
par définition de $T\wedge\alpha$, $T\wedge\di\alpha$
et du calcul différentiel sur les courants, on a
\[ \langle T\wedge\di\alpha, \rho\omega\rangle
= \int_{P_f^T} \di\alpha\wedge\rho\omega \]
et
\[ \langle \di(T\wedge\alpha), \rho\omega\rangle
= (-1)^{p+q+1}\langle T\wedge\alpha,\di(\rho\omega) \rangle 
= (-1)^{p+q+1}\int_{P_f^T} \alpha\wedge\di(\rho\omega) .\]
Il s'agit de prouver que
\[ 0 = \int_{P_f^T} (\di \alpha \wedge \rho\omega + (-1)^{p+q} \alpha \wedge\di(\rho\omega)) 
= \int_{P_f^T} \di(\alpha \wedge \rho\omega) . \]

En remplaçant $\alpha$ par~$\alpha\wedge\omega$ 
et en adjoignant~$f_\trop$ à~$u$, on se ramène
au cas où $\omega=1$;
la forme $\alpha$ est alors de type~$(p,p)$.
Le courant~$T$ est de type~$(p,p)$
et $P_f^T$ est purement de dimension~$p$.
Notons $\beta$ une forme lisse au voisinage de~$u(W)$ telle
que $\alpha=u^*\beta$.

Notons~$\Sigma$ un voisinage paralinéaire compact
de~$P_f^T\cap \supp(\rho)$ dans~$P_f^T$.
Soit $\mathscr C$ une décomposition cellulaire de~$\Sigma$
qui est adaptée à son $p$-calibrage
et telle les~$u_i$ soient affines sur chaque cellule.
Pour toute cellule~$C$ de dimension~$p$, notons~$\mu_C$
son $p$-calibrage.

Par la formule de Stokes, on a 
\[
\int_{P_f^T} \di(\rho\alpha) = 
\int_{\Sigma} \di(\rho u^*\beta) = 
 \int_{\Sigma}^\partial \rho u^*\beta
 = \sum_{D\in\mathscr C_{p-1}}
\int_D \left(
  \sum_{C\supset D} \langle (\rho u^*\beta)_C, \mu_C\rangle\right). \]

Soit $D$ une $(p-1)$-cellule. 
Si $C$ est une cellule contenant~$D$ telle que l'application
affine~$u|_C$ ne soit pas injective, on a $u^*\beta=0$ sur~$C$.
Si~$\mathscr C_D$ est l'ensemble des cellules~$C\in\mathscr C$ 
de dimension~$p$ qui contiennent~$D$ et telles que l'application affine~$u|_C$
soit injective, on a donc
\[
 \int_\Sigma^\partial \rho u^*\beta
 = \sum_{D\in\mathscr C_{p-1}}
\int_D \left(
  \sum_{C\in \mathscr C_D} \langle (\rho u^*\beta)_C, \mu_C\rangle\right). \]
Si $u|_D$ n'est pas injective, alors $u|_C$ n'est injective 
pour aucune cellule~$C$ contenant~$D$, donc $\mathscr C_D$ est vide.

Supposons que $u|_D$ soit injective et notons~$u_D^{-1}$
la bijection réciproque de $u_D\colon D\to u(D)$.
Soit  $C\in\mathscr C_D$;
notons $u_C^{-1}$ la bijection
réciproque de~$u|_C \colon  C \to u(C)$; sa restriction à~$u(D)$
coïncide avec~$u_D^{-1}$.
Sur $D$, on a l'égalité de $(p-1,0)$-formes 
\[ \langle (\rho u^*\alpha)_C, \mu_C \rangle
= \rho  u_C^* \langle \alpha_{u(C)} , u_*\mu_C\rangle ,  \]
donc 
\[ \int_D \langle (\rho u^*\alpha)_C, \mu_C \rangle
= \int_{D} \rho u_C^* \langle \alpha_{u(C)}, u_*\mu_C\rangle
= \int_{u(D)} \rho\circ u_D^{-1} \langle \alpha_{u(C)}, u_*\mu_C\rangle. 
\]
La contraction d'une $(p-1,p)$-forme sur~$\R^p$
avec un $p$-vecteur-volume est une opération ponctuelle;
comme $\alpha$ est supposée lisse, 
on peut écrire 
\[ \langle \alpha_{u(C)}, u_*\mu_C\rangle
= \langle \alpha, u_*\mu_C\rangle |_C. \]
Alors
\[ \int_{u(D)} \rho\circ u_D^{-1} \langle \alpha_{u(C)}, u_*\mu_C\rangle
= \int_{u(D)} \rho\circ u_D^{-1} \langle \alpha, u_*\mu_C\rangle  \]
et
\[ \int_{D} \sum_{C\in\mathscr C_D}
\langle (\rho u^*\alpha)_C, \mu_C \rangle
= \int_{u(D)} \rho\circ u_D^{-1} \langle \alpha, \sum_{C\in\mathscr C_D} u_*\mu_C\rangle. \]

Il suffit donc de prouver l'égalité
\[ \sum_{C\in\mathscr C_D} u_*\mu_C = 0  \]
pour toute cellule~$D$ de dimension~$(p-1)$ de~$\mathscr C$.

Utilisons l'hypothèse que toute composante~$v$
de~$u$ est harmonique: comme $T$ est un courant de Bedford--Taylor,
on a donc $\ddc(vT)=0$ sur~$\mathring W$.
Si $\omega$ est une forme lisse $f$-tropicale 
et $\rho$ est une fonction lisse
telles que $\rho\omega$ soit à support compact dans~$\mathring W$, 
la proposition~\ref{prop.ddcuT} fournit l'égalité  
\[ 0 = \langle \ddc(vT),\rho\omega\rangle
= \int_{\Sigma} \ddc v \wedge\rho\omega
- \int_\Sigma^\partial \dc v \wedge \rho\omega
= - \int_\Sigma^\partial \dc v \wedge \rho\omega \]
puisque $v$ est paralinéaire.

Fixons une $(p-1)$-cellule~$D\in\mathscr C$
et soit $x\in\mathring D$.
Considérons un morphisme de tores $p\colon T\to \gm^{p-1}$
tel que la composée $(p\circ f)_\trop $
induise une immersion par morceaux d'un voisinage
de~$x$ dans~$D$ vers~$\R^{p-1}$.

Dans la suite, on va se contenter de considérer des formes~$\omega$
du type $f^*p^*\gamma$, où 
$\gamma$ est une $(p-1,p-1)$-forme lisse sur~$\R^{p-1}$,
et des fonctions lisses~$\rho$ 
dont le support ne rencontre aucune $p$-cellule de~$\Sigma$
autres que celles qui contiennent~$D$.

Choisissons un $(p-1)$-vecteur $\mu_D$, une orientation~$\eps_D$.
Pour toute $p$-cellule de~$\mathscr C$ contenant~$D$,
on en déduit une orientation de~$C$ (par la règle de la normale sortante)
et un vecteur~$e_C$ dans~$C$ tel que le 
calibrage canonique de~$C$ soit égal à $[e_C\wedge \mu_D, \eps_C]$.
Fixons aussi une fonction affine~$\phi_C$ sur~$C$ qui est nulle sur~$D$
et telle que $\phi_C(e_C)=1$.


Soit $a_C=\partial_{e_C}(v|_C)$ la dérivée de~$v|_C$ dans la direction de~$e_C$;
alors $v_C=v-a_C \phi_C$ est invariant par~$e_C$ et donc provient de~$D$
par la projection de~$C$ sur~$D$ parallèlement à~$e_C$.
Sur $C$, on a donc $\dc v= \dc v_C+ a_C \dc\phi_C$.
Compte tenu de la construction de~$\omega=f^*p^*\gamma$,
lorsqu'on calcule le produit extérieur $\dc v\wedge \omega$,
qu'on le contracte par le calibrage canonique~$\mu_C$
et qu'on le restreint à~$D$, il reste
\[  a_C \theta \langle  f^*p^*\gamma|_D, \mu_D\rangle. \]
On a donc
\[ \int_D \sum_{C\supset D} a_C \rho \langle f^*p^*\gamma, \mu_D\rangle = 0. \]
Comme $\theta$ et $\beta$ sont arbitraires, on en déduit
$ \sum_{C\supset D} a_C f_*p^*\mu_D = 0$;
en simplifiant par le $(p-1)$-vecteur non nul $p_*f_*\mu_D$, on a ainsi
prouvé l'égalité
\[ \sum_{C\supset D} \partial_{e_C} (v|_C) = 0. \]

Avec les notations précédentes,
$\mu_C=\mu_D\wedge e_C$ et $u_*(\mu_C)=u_*(\mu_D)\wedge u_*(e_C)$.
Le premier facteur $u_*(\mu_D)$ est indépendant de~$C$,
le second est le vecteur dont les coordonnées
sont les $a_C \phi_C(e_C)= \partial_{e_C}(v|_C)$, 
lorsque $v$ parcourt les coordonnées de~$u$.
On a donc
\[ \sum_{D\supset C} u_*(\mu_C) = u_*(\mu_D) \wedge \sum_{C\supset D} u_*(e_C)
= 0 \]
puisque les composantes du vecteur $\sum_C u_*(e_C)$ sont données
par $\sum_C \partial_{e_C}(v|_C)=0$.
\end{proof}

\begin{coro}\label{coro.produit-BT=produit-G}
Soit $T$ un courant squelettique sur~$X$,
localement différence de courants de Bedford--Taylor,
et soit $u_1,\dots,u_d$ des fonctions G-lisses,
localement harmoniquement tropicalisables.
Alors le produit
\[ T \wedge \ddc (u_1) \wedge \dots \wedge \ddc (u_d) \]
pris au sens de la théorie de Bedford--Taylor coïncide avec le produit
du courant squelettique~$T$ par la G-forme
$\ddc(u_1)\wedge \dots \wedge \ddc(u_d)$.
\end{coro}
\begin{proof}
Pour résoudre le conflit de notations, désignons dans cette
preuve par $\cap$ le produit à la Bedford--Taylor
et par $\wedge$ le produit naïf entre G-formes,
ou celui d'un courant paralinéaire par une G-forme.
Il s'agit de prouver
\[ T \cap  \ddc (u_1) \cap \dots \cap \ddc (u_d) 
= T \wedge \ddc (u_1) \wedge \dots \wedge \ddc (u_d) .\]
Par récurrence, on peut supposer que
\[ T \cap  \ddc (u_1) \cap \dots \cap \ddc (u_{d-1}) 
= T \wedge \ddc (u_1) \wedge \dots \wedge \ddc (u_{d-1}) .\]
Ce courant~$S$ est localement différence de courants de Bedford--Taylor;
il s'agit alors de démontrer que $S\cap \ddc(u_1)=S\wedge \ddc(u_1)$. On est
ainsi ramené au cas $d=1$ et $S=T$. Alors, on a
\[ T \cap \ddc(u_1)=\ddc (u_1 T)= \di \big( \dc (u_1 T))
= \di (T \wedge \dc u_1) = T \wedge \ddc u_1  \]
d'après la proposition~\ref{prop.di-T-wedge=} appliquée avec $\dc$
puis avec $\di$.
Cela conclut la démonstration.
\end{proof}

\begin{coro}\label{coro.wedge-weak-psh}
Soit $X$ un espace tropical de niveau~$n$, sans bord.
Soit $T$ un courant positif fermé sur~$X$
et soit $u$ une fonction continue,
localement de la forme $f^* v$, où $f$ est une carte harmonique
et $v$ est une fonction convexe.
Alors $u$ est dpsh et le courant $T\wedge \ddc u$ sur~$X$ est positif.
\end{coro}
\begin{proof}
Le résultat est local sur~$X$; on peut donc supposer
qu'il existe une telle carte harmonique sur~$X$.
Par approximation, on peut également supposer que $v$ est lisse.
Alors, $u$ est $G$-lisse, de sorte que $u$ est dpsh.
De plus, la G-forme~$\ddc u$ est égale à $f^*(\ddc v)$, donc est positive.
Le corollaire précédent identifie le courant $T\wedge \ddc u$
au produit de~$T$ par la G-forme positive $\ddc u$.
par définition de ce dernier courant, il est positif.
\end{proof}

\section{Négligeabilité des ensembles polaires}

\begin{theo}\label{theo.trop-psh-locint}
Soit $X$ un espace tropical de niveau~$n$, sans bord.
On suppose que tout point~$x$ de~$X$ possède une base de voisinages ouverts~$U$
sur lesquels il existe une fonction psh lisse $\rho\colon U\to\R$
telle que l'ouvert de~$U$ défini par $\{\rho < 0 \}$ soit 
un voisinage relativement compact de~$x$ dans~$U$.

Soit $u_1,\dots,u_n$ des fonctions psh continues sur~$X$
qui sont localement approchables,
soit $V$ une fonction psh localement approchable sur~$X$.
Si $V$ est localement intégrable sur~$X$, est elle est localement intégrable
pour la mesure $\ddc u_1\wedge\dots\ddc u_n$ sur~$X$.
\end{theo}

\begin{proof}
L'assertion est locale sur~$X$.  
Soit $x$ un point de~$X$.
Comme les fonctions~$u_i$ et la fonction~$V$ sont localement approchables,
on peut donc supposer que les~$u_i$
sont limite uniforme de suites $(u_{i, \ell})_\ell$
et que $V$ est limite décroissante d'une suite $(V_p)_p$ 
de fonctions psh lisses sur~$X$.
Compte tenu de l'hypothèse sur~$X$, on peut en outre
supposer qu'il existe  une fonction psh lisse~$\rho$ sur~$X$
telle que l'ouvert~$U$ défini par $\rho<0$ soit un voisinage
relativement compact de~$x$.

Reprenons les notations de la proposition~\ref{prop.cln-2}
en prenant $m=n$ et $a=3$.
Pour tous entiers~$\ell$ et~$p$, on a donc l'inégalité
\[ \int_U \beta_n \wedge (-V_p) \ddc u_{1,\ell}\wedge \dots\wedge \ddc u_{n,\ell}
\leq C(n,3) \norm{u_{1,\ell}}_U \dots \norm{u_{n,\ell}}_U \int_U (-V_p) \beta_0 . \]
Comme $\beta_n$ est nulle hors de~$U$, on peut récrire cette inégalité
\[ \int_X \beta_n \wedge (-V_p) \ddc u_{1,\ell}\wedge \dots\wedge \ddc u_{n,\ell}
\leq C(n,3) \norm{u_{1,\ell}}_U \dots \norm{u_{n,\ell}}_U \int_U (-V_p) \beta_0 . \]
Faisons tendre $\ell$ vers~$+\infty$;
la proposition~\ref{prop.BT-continu} 
et l'exemple~\ref{prop.courant.positif-mesure} entraînent
\[ \int_X \beta_n \wedge (-V_p) \ddc u_{1}\wedge \dots\wedge \ddc u_{n}
\leq C(n,3) \norm{u_{1}}_U \dots \norm{u_{n}}_U \int_U (-V_p) \beta_0 . \]
Faisons maintenant tendre~$p$ vers~$+\infty$; par convergence monotone,
on obtient
\[ \int_X \beta_n \wedge (-V) \ddc u_{1}\wedge \dots\wedge \ddc u_{n}
\leq C(n,3) \norm{u_{1}}_U \dots \norm{u_{n}}_U \int_U (-V) \beta_0 . \]
La fonction~$\beta_n$ est continue et strictement positive sur~$U$.
Par suite, $(-V)$ est localement intégrable sur~$U$ 
pour la mesure $\ddc u_1\wedge\dots\wedge\ddc u_n$.
\end{proof}

\begin{coro}\label{coro.local-int}
Soit $X$ un espace tropical comme dans le théorème~\ref{theo.psh-locint}.
Soit $u_1,\dots,u_m$ des fonctions psh continues sur~$X$
qui sont localement approchables,
soit $\alpha$ une forme lisse de type~$(n-m,n-m)$ sur~$X$,
et soit $V$ une fonction psh localement approchable sur~$X$.
Si $V$ est localement intégrable sur~$X$, est elle est localement intégrable
pour la mesure $\ddc u_1\wedge\dots\ddc u_m\wedge\alpha$ sur~$X$.
\end{coro}
\begin{proof}
L'assertion est locale sur~$X$. 
On peut donc supposer que la forme~$\alpha$ est tropicale, 
du type $f^*\omega$, où $f\colon X\to\R^N$ est une carte affine. 
Soit $x\in X$. Pour vérifier la propriété d'intégrabilité
au voisinage de~$x$, on peut remplacer~$\alpha$ par $\rho\alpha$,
où $\rho$ est une fonction lisse à support compact, non nulle en~$x$.
Soit $v=f^*\norm{t}^2$; il existe alors un nombre réel~$c$ tel que
$- c (\ddc v)^{n-m}\leq\alpha\leq c(\ddc v)^{n-m}$.
Le théorème précédent, appliqué aux fonctions $u_1,\dots,u_m,v,\dots,v$,
entraîne le résultat voulu.
\end{proof}

\begin{coro}\label{coro.local-polaire}
Soit $X$ un espace tropical comme dans le théorème~\ref{theo.trop-psh-locint}.
Soit $u_1,\dots,u_m$ des fonctions psh continues sur~$X$
qui sont localement approchables,
soit $\alpha$ une forme lisse de type~$(n-m,n-m)$ sur~$X$.
et soit $V$ une fonction psh localement approchable sur~$X$.
Si $V$ est localement intégrable sur~$X$, 
la mesure $\ddc u_1\wedge\dots\ddc u_m\wedge\alpha$ sur~$X$
ne charge pas le lieu polaire de~$V$.
\end{coro}

\chapter{Intersection tropicale}
\section{Produit d'espaces tropicaux}
\label{sect.prod-trop}

Soit $X$ et $Y$ des espaces tropicaux.
Le but de ce paragraphe est de définir une structure d'espace
tropical sur le produit $X\times Y$.
On commence par munir ce produit d'une structure d'espace G-tropical.

\begin{lemm}
Soit $X$ et $Y$ des espaces G-tropicaux.
L'espace topologique produit est séparé, localement compact
et chacun de ses points a une base d'ouverts qui sont dénombrables à l'infini.
\end{lemm}
\begin{proof}
Comme $X$ et $Y$ sont séparés et localement compacts, il en est
de même de leur produit.
Soit $(x,y)\in X\times Y$ et soit $W$ un  voisinage de~$(x,y)$
dans~$X\times Y$. Par définition de la topologie produit,
il existe un voisinage ouvert~$U$ de~$x$ dans~$X$ 
et un voisinage ouvert~$V$ de~$y$ dans~$Y$ tels que $U\times V\subset W$.
Par hypothèse, $x$ possède un voisinage ouvert~$U_1$ contenu dans~$U$
qui est réunion d'une famille dénombrable $(K_i)_{i\in I}$ 
de parties compactes;
de même, $y$ possède un voisinage ouvert~$V_1$ contenu dans~$U$
qui est réunion d'une famille dénombrable $(L_j)_{j\in J}$ 
de parties compactes.
Alors $U_1\times V_1$ est un voisinage ouvert de~$(x,y)$ contenu
dans~$W$ qui est réunion de la famille dénombrable $(K_i\times L_j)_{(i,j)\in I\times J}$ de parties compactes.
\end{proof}

\subsection{}
Soit $X$ et~$Y$ des espaces G-tropicaux.
On dit qu'une partie localement fermée~$D$
de~$X\times Y$ est un domaine élémentaire s'il existe
une carte tropicale $(V,f,P)$ sur~$X$ (où $P$ est fermé), 
une carte tropicale $(W,g,Q)$ (où $Q$ est fermé)
sur~$Y$ et une partie paralinéaire fermée~$R$ de~$P\times Q$
telles que $D=(f\times g)^{-1}(R)$.
On dit qu'une partie localement fermée de~$X\times Y$ est un domaine
si elle est G-recouverte par ses domaines élémentaires.

Toute partie ouverte est un domaine.

L'intersection de deux domaines élémentaires est un domaine élémentaire, 
donc l'intersection de deux domaines est un domaine.

Toute partie localement fermée de~$X\times Y$
qui est G-recouverte par des domaines
est encore un domaine.

Tout domaine élémentaire est G-recouvert par des domaines élémentaires
compacts, si bien que tout domaine est 
G-recouvert par les domaines compacts qu'il contient.

Les axiomes sur les domaines sont donc vérifiés.

\subsection{}
On dit qu'une fonction~$\phi$ définie sur un domaine élémentaire
de la forme $V\times W$ comme ci-dessus est élémentairement paralinéaire
s'il existe une fonction paralinéaire~$h$ sur~$R$ telle
que $\phi(x,y)=h(f(x),g(y))$ pour tout $(x,y)\in V\times W$.

On dit qu'une fonction~$\phi$ définie sur un domaine de~$X\times Y$ 
est paralinéaire si ce domaine est G-recouvert par des domaines
élémentaires sur lesquels $\phi$ est élémentairement paralinéaire.

Les axiomes sur les fonctions paralinéaires sont vérifiés.

\subsection{}
Soit $X$ et $Y$ des espaces G-tropicaux affinés ;
construisons un affinage de l'espace G-tropical $X\times Y$.
On dit qu'une fonction paralinéaire~$h$ sur un domaine~$W$ de~$X\times Y$
est affine 
si pour tout point $(x,y)\in W$, il existe un domaine~$U$ de~$X$
et un domaine~$V$ de~$Y$ 
tels $(U\times V)\cap W$ soit un voisinage de~$(x,y)$ dans~$W$,
une fonction affine~$f$ sur~$U$ et
une fonction affine~$g$ sur~$V$
telles que l'on ait 
$h(\xi, \eta)=f(\xi)+g(\eta)$ 
pour tout $(\xi,\eta)\in  (U\times V)\cap W$.

Une fonction $h$ sur un domaine~$W$ de~$X\times Y$
est paralinéaire si et seulement si elle est G-localement affine.
Il suffit de traiter le cas des fonctions élémentairement paralinéaires.
Cela résulte alors de ce qu'une fonction paralinéaire
sur une partie paralinéaire de~$\R^n$ est G-localement affine.

\subsection{}
Supposons que $X$ et~$Y$ soient chordés, c'est-à-dire
munis d'une classe de parties squelettiques élémentaires vérifiant les 
axiomes du~\S\ref{defi.squel-elem}.
On dit alors qu'une partie fermée d'un domaine~$W$ de~$X\times Y$ 
est squelettique
élémentaire s'il existe des domaines~$U$ de~$V$ de~$X$ et~$Y$ respectivement,
une partie squelettique élémentaire~$\Sigma$ de~$U$,
une partie squelettique élémentaire~$\Tau$ de~$V$,
telle que la partie donnée soit égale à $(\Sigma\times\Tau)\cap W$.
Démontrons que l'on munit ainsi l'espace~$X\times Y$ 
d'une structure d'espace chordé.

Les axiomes~\emph a) et \emph b) de~\loccit résultent de la construction.

Démontrons l'axiome~\emph c).
Soit $(W_i)$ une famille finie de domaines fermés de~$X$;
pour tout~$i$, soit $\Upsilon_i$ une partie squelettique élémentaire de~$W_i$
et soit $\Upsilon$ la réunion des~$\Sigma_i$.
Pour tout~$i$, soit $U_i$ un domaine de~$X$, $\Sigma_i$ une partie
squelettique élémentaire de~$U_i$, $V_i$ un domaine de~$Y$,
$\Tau_i$ une partie squelettique élémentaire de~$V_i$
telles que $\Upsilon_i=(\Sigma_i\times\Tau_i) \cap W_i$.

L'assertion est G-locale, et donc locale, sur~$\Upsilon$.
Soit $(x,y)\in \Upsilon$.
Les espaces~$X$ et~$Y$ sont localement compacts
et 
les parties~$\Upsilon_i$ sont localement fermées dans~$W_i$,
donc dans~$X\times Y$; 
on peut donc remplacer~$X$ et~$Y$ par des voisinages domaniaux compacts
de~$x$ et~$y$ de sorte que $(x,y)\in \Upsilon_i$ pour tout~$i$;
comme les domaines~$W_i$ sont fermés dans~$X\times Y$,
ils sont donc compacts,
de même que leurs parties squelettiques élémentaires~$\Upsilon_i$.
Alors, $x\in U_i$ et $y\in V_i$ pour tout~$i$;
comme les~$U_i$ sont des parties localement fermées de~$X$ et
les~$V_i$ sont des parties localement fermées de~$Y$,
on peut restreindre de nouveau~$X$ et~$Y$ et supposer que les~$U_i$
et les~$V_i$ sont compacts.
Soit $\Sigma$ la réunion des~$\Sigma_i$ et soit $\Tau$ la réunion de~$\Tau_i$.
D'après l'axiome~\emph c) appliquée à la famille
finie des parties squelettiques élémentaires $\Sigma_i$
des domaines fermés~$U_i$ de~$X$,
il existe une famille $(U'_j)$ de domaines de~$X$ qui G-recouvre~$\Sigma$
et, pour tout~$j$, une partie squelettique élémentaire~$\Sigma'_j$
de~$U'_j$ qui contient~$\Sigma\cap U'_j$.
De même, il existe une famille $(V'_k)$ de domaines de~$Y$ qui G-recouvre~$\Tau$
et, pour tout~$k$, une partie squelettique élémentaire~$\Tau'_k$
de~$V'_k$ qui contient~$\Tau\cap V'_k$.

Pour tout couple $(j,k)$, posons $W'_{j,k}=U'_j\times V'_k$;
les $W'_{j,k}$ sont des domaines de~$X\times Y$ qui
G-recouvrent $\Sigma\times\Tau$, donc~$\Upsilon$.
Par ailleurs
$\Sigma'_j\times\Tau'_k$ est une partie squelettique élémentaire
de~$W'_{j,k}$ qui contient $(\Sigma\cap U'_j)\times(\Tau\cap V'_k)
=(\Sigma\times\Tau)\cap W'_{j,k}$,
donc contient $\Upsilon\cap W'_{j,k}$.

\begin{lemm}
Soit $X$ et $Y$ des espaces G-tropicaux vertébrés de niveaux respectifs~$m$
et~$n$. 
L'espace G-tropical chordé~$X\times Y$ est un
espace G-tropical vertébré de niveau~$m+n$.
\end{lemm}
\begin{proof}
Il s'agit de démontrer \emph a) que la dimension tropicale de~$X\times Y$
est au plus égale à $m+n$ et \emph b) que pour toute carte
G-tropicale $(W,h)$ sur~$X$, la partie $\Sigma_h^{(m+n)}$ 
est squelettique.

\begin{enumerate}
\item
Démontrons que $d_\trop(X\times Y)\leq m+n$.
Soit $u$ une carte tropicale sur un domaine compact~$W$ de~$X\times Y$;
démontrons que $u(W)$ est une partie paralinéaire de dimension
au plus~$m+n$.
L'assertion est G-locale sur~$W$;
on peut ainsi supposer
que $W$ est un domaine élémentaire compact 
sur lequel $u$ est élémentairement paralinéaire.
On se ramène ainsi au cas où il existe un domaine compact~$U$ de~$X$, 
un domaine compact~$V$ de~$Y$, 
des cartes G-tropicales~$f\colon U\to E$ et~$g\colon V\to F$, 
une partie paralinéaire fermée~$R$ de~$f(U)\times g(V)$
et une application paralinéaire $h\colon R\to\R^p$
tels que $W=(f,g)^{-1}(R)$ et $u = h\circ (f ,g)$.
On a $u(W)\subset h((f,g)(U\times V))=h(f(U)\times g(V))$.
Par hypothèse, $f(U)$ et $g(V)$ sont des parties paralinéaires
de dimensions au plus~$m$ et~$n$ respectivement;
par suite, $h(f(U)\times g(V))$ est une partie paralinéaire
de dimension au plus~$m+n$, et il en est donc de même de~$u(W)$.

\item
Soit $(W,h)$ une carte G-tropicale sur~$X$.
Par hypothèse, $W$ est G-recouvert par ses domaines élémentaires
de la forme $W_i=(U_i\times V_i)\cap (f_i\times g_i)^{-1}(R_i)$,
où, pour tout~$i$,
$(U_i,f_i,P_i)$ est une carte G-tropicale sur~$X$ (où $P_i$ est fermé)
$(V_i,g_i,Q_i)$ est une carte G-tropicale sur~$Y$ (où $Q_i$ est fermé),
$R_i$ est une partie paralinéaire fermée de~$P_i\times Q_i$,
$h|_{W_i}=\phi_i(f_i,g_i)$, 
et $\phi_i$ est une fonction paralinéaire sur~$P_i\times Q_i$.
Quitte à raffiner ce G-recouvrement,
on suppose que les~$U_i$ sont fermés dans~$X$,
les~$V_i$ sont fermés dans~$Y$ et
les applications paralinéaires~$\phi_i$ sont en fait affines;
on pose alors $h_i=\phi_i(f_i,g_i)$ sur $U_i\times V_i$.

L'assertion à démontrer est locale sur~$W$.
Soit $(x,y)\in W$.
Par définition d'un G-recouvrement, 
il existe un voisinage ouvert~$\Omega$ de~$(x,y)$ dans~$X\times Y$ et
un ensemble fini~$I$ d'indices tels que
$W_i\cap \Omega$ soit fermé dans~$W$ pour tout~$i$
et $\bigcup_{i\in I} W_i\cap\Omega = W\cap\Omega$.
Il suffit de prouver que $\Sigma_h^{(m+n)}\cap\Omega$ est une partie
paralinéaire d'une une partie squelettique élémentaire de~$\Omega$.
Quitte à restreindre~$\Omega$, on peut supposer que c'est le produit
d'un voisinage de~$x$ et d'un voisinage de~$y$,
si bien qu'on peut supposer que $\Omega=X\times Y$.
On a l'inclusion
\[ \Sigma_h^{(m+n)}
 \subset \bigcup_{i\in I} \Sigma_{h_i}^{(m+n)}. \]

Pour tout $i$, on $h_i=\phi_i(f_i,g_i)$, d'où l'inclusion
\[ \Sigma_{h_i}^{(m+n)}\subset \Sigma_{(f_i,g_i)}^{(m+n)}
=\Sigma_{f_i}^{(m)}\times \Sigma_{g_i}^{(n)}, \]
si bien que 
\[ \Sigma_h^{(m+n)} \subset \bigcup_{i\in I} 
(\Sigma_{f_i}^{(m)}\times \Sigma_{g_i}^{(n)}) \cap (f_i,g_i)^{-1}(R_i). \]
Notons~$\Sigma$ le terme de droite.
Pour tout~$i$, $\Sigma_{f_i}^{(m)}\times \Sigma_{g_i}^{(n)}$
est une partie squelettique (élémentaire) de~$U_i\times V_i$
sur laquelle $(f_i,g_i)$ est une application paralinéaire;
par suite, 
$(\Sigma_{f_i}^{(m)}\times \Sigma_{g_i}^{(n)}) \cap (f_i,g_i)^{-1}(R_i)$
est une partie squelettique fermée de~$U_i\times V_i$,
donc de~$X\times Y$.
Par suite, $\Sigma$ est une réunion finie de parties squelettiques fermées,
donc est squelettique.

Il suffit maintenant de démontrer que $\Sigma_h^{(m+n)}$
en est une partie paralinéaire.

Nous allons démontrer que $\Sigma_h^{(m+n)}$
est le lieu $(m+n)$-dimensionnel~$\Sigma'$ de $h|_\Sigma$.

On a $\Sigma'\subset \Sigma_h^{(m+n)}$.

Démontrons l'inclusion réciproque.
Soit $(\xi,\eta)\in\Sigma_h^{(m+n)}$;
il existe $i\in I$ tel que $(\xi,\eta)\in U_i\times V_i$
et $d_\trop(h|_{W_i} )=m+n$.
Démontrons qu'alors que la restriction de~$h$ à 
$({\Sigma_{f_i}^{(m)}\times\Sigma_{g_i}^{(n)}}) \cap (f_i,g_i)^{-1}(R_i)$ 
est de dimension tropicale~$m+n$ en~$(\xi,\eta)$.

On est donc ramené au cas où $X=U_i$ et $Y=V_i$;
on note ainsi $f=f_i$, $g=g_i$, $\phi=\phi_i$ et $R=R_i$;
on a donc $h=\phi (f,g)$ et $W=(f,g)^{-1}(R)$.

La restriction de~$f$ à~$\Sigma_f^{(m)}$ est une immersion par morceaux,
et celle de~$g$ à~$\Sigma_g^{(n)}$ est une immersion par morceaux,
de sorte que la restriction de~$(f,g)$ à~$\Sigma_f^{(m)}\times\Sigma_g^{(n)}$,
est encore une immersion par morceaux,
de même que sa restriction à~$\Sigma$.

Soit~$Z$ un voisinage domanial compact  de $(\xi,\eta)$
dans~$X\times Y$; démontrons que $h(Z\cap \Sigma)$ est
de dimension~$m+n$.
Comme $Z$ est G-recouvert par des voisinages domaniaux compacts
produit, on peut le supposer de la forme $K\times L$,
où $K$ et~$L$ sont des voisinages domaniaux compacts de~$\xi$ et~$\eta$
dans~$X$ et~$Y$ respectivement.
 
Par hypothèse, $h((K\times L) \cap W)$ est de dimension~$m+n$.
On a 
\[ h((K\times L) \cap W) 
 = \phi(f,g) ((K \times L) \cap (f,g)^{-1}(R)) 
 = \phi ( (f(K)\times g(L)) \cap R ) . \]
Considérons des décompositions cellulaires~$\mathscr C$ et~$\mathscr D$
de~$f(K)$ et~$g(L)$ respectivement. 
Il existe donc une cellule $C\in\mathscr C$ et une cellule $D\in\mathscr D$
telles que $\phi((C\times D)\cap R)$ soit de dimension~$m+n$.
Comme $\dim(C)\leq m$ et $\dim(D)\leq n$, cela entraîne
que l'on a $\dim(C)=m$, $\dim(D)=n$, et l'application affine~$\phi$ 
est injective sur le sous-espace $\langle C\rangle \times \langle D\rangle$.
De plus, $(C\times D)\cap R$ est de dimension~$m+n$;
il existe donc une $m$-cellule~$C'$ contenue dans~$C$ et
une $n$-cellule~$D'$ contenue dans~$D$ telles que $C'\times D'\subset R$;
notons que la restriction de~$\phi$ à~$C'\times D'$ est encore injective.

Posons $W'= (K\times L) \cap (f,g)^{-1}(C'\times D') $; 
c'est un domaine compact de~$W$;  
démontrons que $h(W'\cap \Sigma)$ est de dimension~$m+n$.
On a
\[ h(W'\cap \Sigma) = \phi (
   (f,g) (K\times L \cap \Sigma) \cap (C'\times D')  ). \]
Nous allons prouver que 
$ (f,g) (K\times L \cap \Sigma) $
contient $C'\times D'$; 
il en résultera que $h(W'\cap\Sigma)$ contient $\phi(C'\times D)$
dont on a prouvé qu'il est dimension~$m+n$, d'où l'assertion.

Il reste à établir que $(f,g)(K\times L\cap \Sigma)$ contient $C'\times D'$.
Observons que $C'\subset f(K)$ et $D'\subset g(L)$.
D'après la proposition~\ref{prop.image-squelette-pur},
on a $f(K)^{(m)} = f(\Sigma_{f|_K}^{(m)})$
et $g(L)^{(n)}=g(\Sigma_{g|_L}^{(n)})$.
Par conséquent, 
\begin{align*}
 C'\times D' & \subset f(K)^{(m)}\times g(L)^{(n)} \\
& \subset  f(\Sigma_{f|_K}^{(m)}) \times g(\Sigma_{g|_L}^{(n)}) \\
& \subset  f(\Sigma_{f}^{(m)} \cap K) \times g(\Sigma_{g}^{(n)} \cap L) \\
& = (f,g) ((\Sigma_{f}^{(m)} \cap K) \times (\Sigma_{g}^{(n)} \cap L))
, \end{align*}
donc 
\begin{align*} 
 C'\times D'  & \subset 
  (f,g) ((\Sigma_{f}^{(m)} \times \Sigma_{g}^{(n)}) \cap (K\times L)) \cap R \\
& =
   (f,g) ( (\Sigma_{f}^{(m)} \times \Sigma_{g}^{(n)}) \cap (K\times L) \cap (f,g)^{-1}(R) 
) \\
& = (f,g) ((K \times L) \cap \Sigma ). \end{align*}
Le lemme est ainsi démontré.
\qedhere
\end{enumerate}
\end{proof}

\subsection{}
Supposons que $X$ et~$Y$ sont des espaces G-tropicaux
calibrés de niveaux~$m$ et~$n$ respectivement
et munissons $X\times Y$ d'une structure d'espace G-tropical
calibré de niveau~$m+n$.
D'après le lemme~\ref{lemm.gros-site},
il suffit de construire une famille compatible
de calibrages sur les parties squelettiques de~$X\times Y$
de la forme  $\Sigma\times\Tau$, où $\Sigma$ est une partie
squelettique de~$X$ et $\Tau$ est une partie squelettique de~$Y$,
et la famille des calibrages produit convient.

\subsection{}
Supposons maintenant que $X$ et~$Y$ sont des espaces tropicaux
de niveaux respectifs~$m$ et~$n$,
d'intérieurs respectifs~$\Int(X)$ et~$\Int(Y)$.

On définit pour bord de~$X\times Y$ 
la partie $(\partial(X)\times Y)\cup (X\times\partial(Y))$,
de sorte que l'intérieur de~$X\times Y$ est 
sa partie $\Int(X)\times \Int(Y)$.

Il s'agit de démontrer (i) que cet intérieur est
localement différentiellement séparé, décent, très riche et harmonieux,
et (ii) que pour toute partie squelettique~$\Sigma$ de~$X\times Y$
et tout domaine~$W$ de~$X\times Y$,
la trace de~$\partial(X)$ et du bord topologique~$\partial(W/X\times Y)$
sur~$\Sigma$ sont $(m+n)$-négligeables.

(i) Pour prouver~(i), on peut remplacer~$X$ et~$Y$ par leurs intérieurs,
et donc les supposer sans bord.

Démontrons que $X\times Y$ est localement différentiellement séparé.
Il suffit de prouver que si $X$ et~$Y$ sont différentiellement séparés,
alors $X\times Y$ est différentiellement séparé.
Soit $(x,y)$ et $(x',y')$ des points distincts de~$X$ et~$Y$,
soit $f$ une fonction lisse sur~$X$ qui prend des valeurs distinctes
en~$x$ et~$x'$ s'ils sont distincts et
soit $g$ une fonction lisse sur~$Y$ qui prend des valeurs distinctes
en~$x$ et~$x'$ s'ils sont distincts.
Les couples $(f(x), g(y))$ et $(f(x'), g(y'))$ sont distincts,
et il existe une application linéaire $\phi\colon\R^2\to\R$
qui les sépare.
La fonction $\phi(f,g)$ sur~$X\times Y$ est lisse ; 
elle sépare $(x,y)$ et~$(x',y')$.

Démontrons que $X\times Y$ est décent.

a) Soit $W$ un domaine de~$X\times Y$ et soit $(x,y)$  un point de~$W$.
Soit $h$ une fonction affine sur~$W$, à valeurs réelles.
Par définition d'une fonction affine sur~$W$, il existe
un domaine~$U$ de~$X$ et un domaine~$V$ de~$Y$
tels que $(U\times V)\cap W$ soit un voisinage de~$(x,y)$ dans~$W$,
une fonction affine~$f$ sur~$U$ et une fonction affine~$g$ sur~$V$
telles que $h$ coïncide avec~$f+g$ sur~$(U\times V)\cap W$.
Comme $X$ est décent en~$x$, il existe un voisinage ouvert~$U'$
de~$x$ dans~$X$ et une fonction affine~$f'$ sur~$U'$
tels que $f$ coïncide avec~$f'$ sur $U\cap U'$.
De même, comme $Y$ est décent en~$y$, 
il existe un voisinage ouvert~$V'$
de~$y$ dans~$Y$ et une fonction affine~$g'$ sur~$V'$
tels que $g$ coïncide avec~$g'$ sur $V\cap V'$.
Alors $U'\times V'$ est un voisinage ouvert de~$(x,y)$ dans~$X\times Y$, 
la fonction $f'+g'$ est une fonction affine $U'\times V'$
qui coïncide avec $f+g$ sur le voisinage
\[ ((U'\cap U) \times (V'\cap V)) \cap W
= (U'\times V')\cap (U\times V)\cap W \]
de~$(x,y)$ dans~$W$.

b)  Soit $W$ un domaine de~$X\times Y$ et soit $(x,y)$ un point de~$W$.
Démontrons qu'il existe un voisinage ouvert~$W'$ de~$(x,y)$ dans~$X\times Y$
tel que $W\cap W'$ soit défini dans~$W'$ par conjonction et disjonction
d'inégalités larges entre fonctions affines sur~$W'$.
Cette propriété est stable par réunion finie et ne dépend
que du germe de~$W$ en~$(x,y)$; 
par définition des domaines de~$X\times Y$,
on peut donc supposer que $W$ est un domaine élémentaire,
c'est-à-dire est de la forme
$(f,g)^{-1}(R)$, où $(U,f,P)$ et $(V,g,Q)$ sont des cartes
tropicales sur des domaines~$U$ et~$V$ de~$X$ et~$Y$ respectivement,
$P$ et~$Q$ étant fermés, et $R$ est une partie paralinéaire
fermée de~$P\times Q$.
Par définition, $R$ est défini dans~$P\times Q$ par conjonction 
et disjonction d'inégalités larges entre fonctions affines.

Comme $X$ est décent en~$x$, on peut supposer que $f$ 
est la restriction à~$U$ d'une fonction affine, toujours notée~$f$,
définie au voisinage de~$x$.
De même, on peut supposer que $g$ est la restriction à~$V$ 
d'une fonction affine, toujours notée~$g$, définie au voisinage de~$y$.

Il existe également une carte tropicale $(U',f',P')$  définie
au voisinage de~$x$ dans~$X$  telle que $f$ soit définie sur~$U'$,
et une partie paralinéaire fermée~$P''$ de~$P'$
telle que $U = (f')^{-1}(P'')$.
On introduit de même $(V',g',Q')$ pour le facteur~$Y$.
On a  alors
\[ W  \cap (U'\times V') = (f,g)^{-1} (R) 
     \cap (f')^{-1}(P'') \cap (g')^{-1}(Q''), \]
ce qui prouve que $W$ est localement définie par une conjonction
et disjonction d'inégalités larges entre fonctions affines.

Démontrons que $X\times Y$ est riche.

Soit $h\colon W\to E$ une carte G-tropicale
définie sur un domaine compact~$W$ de~$X\times Y$;
démontrons qu'il existe une partie squelettique compacte~$\Sigma$
de~$W$ telle que $h(\Sigma)=h(W)$.
En raisonnant G-localement sur~$W$, on se ramène à démontrer
l'assertion lorsque $W$ est un domaine élémentaire de~$X\times Y$
et $h$ est élémentairement paralinéaire,
c'est-à-dire qu'il existe des domaines compacts~$U$ et~$V$
de~$X$ et~$Y$, des applications paralinéaires~$f$ et~$g$
sur~$U$ et~$V$ et une partie paralinéaire compacte~$R$
de $f(U)\times g(V)$ et une application paralinéaire $\phi\colon R\to E$
telle que $W=(f,g)^{-1}(R)$ et $h=\phi\circ (f,g)|_W$.
Comme $X$ et~$Y$ sont riches,
il existe des parties squelettiques compactes $\Sigma$ de~$U$
et $\Tau$ de~$V$ telles que $f(\Sigma)=f(U)$ et $g(\Tau)=g(V)$.
Alors 
on  a
\begin{align*}
 h(W) & = \phi( (f,g) (W)) = \phi( (f,g)  ((f,g)^{-1}(R))) \\
& = \phi((f(U)\times g(V)) \cap R) = \phi( (f(\Sigma)\times g(\Tau)) \cap R) \\
& = \phi( (f,g) ((\Sigma\times\Tau) \cap (f,g)^{-1}(R))) \\
&= h( (\Sigma\times\Tau) \cap (f,g)^{-1}(R)))). \end{align*}
Par construction,
$\Sigma\times\Tau$ est une partie squelettique de~$U\times V$, 
et sa trace avec $(f,g)^{-1}(R)$ en est une partie paralinéaire
donc est une partie squelettique de~$U\times V$.
Elle est contenue dans~$W$, donc c'est une partie squelettique de~$W$.

Démontrons que $X\times Y$ est très riche.
On vient de démontrer qu'il est riche; il reste  à démontrer
que toute partie squelettique compacte est contenue dans une partie
squelettique purement de dimension~$m+n$.
On peut raisonner G-localement.
Compte tenu de la définition des parties squelettiques,
il suffit de traiter le cas d'une partie squelettique élémentaire.
Soit donc $\Sigma$ et~$\Tau$ des parties squelettiques compactes
de~$X$ et~$Y$ respectivement. Comme $X$ est très riche, il existe
une partie squelettique~$\Sigma'$ de~$X$ 
qui est purement de dimension~$m$
et contient~$\Sigma$; il existe de même une partie squelettique~$\Tau'$ 
de~$Y$ qui est purement de dimension~$n$ et contient~$\Tau$.
Alors $\Sigma'\times \Tau'$ est une partie squelettique 
de~$X\times Y$, purement de dimension~$m+n$, qui contient~$\Sigma\times\Tau$.

Soit $h\colon W\to E$
une carte tropicale sur une partie ouverte~$W$ de~$X\times Y$.
Démontrons que $W$ est harmonieux relativement à~$h$.
L'assertion est locale, ce qui permet de supposer que $X$ et~$Y$
sont dénombrables à l'infini ; alors, $X\times Y $ est dénombrable
à l'infini, et donc paracompact.
On peut également restreindre~$X$ et $Y$ de sorte
que $h$ s'écrit $\phi \circ (f,g)$, où $f\colon X\to A$ et~$g\colon Y\to B$ sont
des cartes tropicales sur~$X$ et~$Y$ et $\phi\colon A\times B\to E$ 
est une application affine.
Soit $\mathscr E$ une décomposition cellulaire de~$\Sigma_h^{(m+n)}$
en $h$-cellules qui est adaptée au calibrage de~$\Sigma_h^{(m+n)}$;
prouvons que $\Sigma_h^{(m+n)}$ est harmonieux le long de 
toute $(m+n-1)$-cellule appartenant à~$\mathscr E$.
Compte tenu de la définition, il suffit de traiter le cas où $h=(f,g)$.
Alors, $\Sigma_h^{(m+n)}={\Sigma_f^{(m)}\times\Sigma_g^{(n)}}$,
comme on l'a vu ci-dessus.
On peut modifier la décomposition cellulaire et la supposer
de la forme $\mathscr C\times\mathscr D$ où $\mathscr C$
et $\mathscr D$ sont des décompositions cellulaires de~$\Sigma_f^{(m)}$
et $\Sigma_g^{(n)}$ respectivement.

Une cellule de $\mathscr C\times \mathscr D$ est de la forme $F\times G$
où $F$ et~$G$ sont des cellules de~$\mathscr C$ et~$\mathscr D$;
elle est de dimension $m+n-1$ si et seulement $(\dim F,\dim G)=(m-1,n)$
ou $(m,n-1)$.
La discordance du calibrage de $\Sigma_h^{(m+n)}$ le long de~$F\times G$
est un $(m+n)$-vecteur $\delta_{F\times G}$ de $A\times B$.
Dans le premier cas, on a $\delta_{F\times G} = \delta_F \otimes g_*(\mu_G)$,
où $\delta_F$ est la discordance du calibrage de~$\Sigma_f^{(m)}$
le long de~$F$, donc est nulle.
Le second cas est similaire.

\subsection{}
Soit $W$ un domaine de~$X\times Y$ et soit $\Upsilon$ une partie squelettique
de~$X$.  Démontrons que la trace de $\partial(X\times Y)$
et celle de $\partial(W/X\times Y)$ sur~$\Upsilon$ sont $m+n$-négligeables.

Comme $X$ et $Y$ sont des espaces tropicaux, leurs domaines sont munis
de bords: le bord d'un domaine~$U$ de~$X$ a été défini
par $\partial(U)=\partial(U/X)\cup (U\cap \partial(X))$.
Définissons de même le bord d'un domaine~$W$ de $X\times Y$ par
la formule
\[ \partial (W) = \partial(W/X\times Y) \cup (W \cap \partial(X\times Y)). \]
Remarquons que si $W$ est de la forme $U\times V$, où $U$ est un domaine
de~$X$ et $V$ est un domaine de~$Y$, alors
\[ \partial(W) = \partial(U)\times V \cup U \times \partial(V). \]
Avec ces notations, il s'agit de démontrer 
et la trace de $\partial(W)$ sur~$\Upsilon$ est $m+n$-négligeable.

Soit $(Z_i)$ une famille de domaines de~$X\times Y$
telle que la famille $(Z_i\cap W\cap\Upsilon)$
soit un G-recouvrement de $W\cap\Upsilon$ et telle que 
que pour tout~$i$, 
la trace de $\partial(W\cap Z_i)$ sur~$\Upsilon$ soit négligeable.
Démontrons que la trace de $\partial (W)$ sur~$\Upsilon$ est négligeable.
Il existe une famille d'ouverts $(O_j)$ dans $X\times Y$,
recouvrant $W\cap\Upsilon$ et, pour tout~$j$,
un ensemble fini~$I_j$ d'indices telle que pour tout $i\in I_j$,
$Z_i\cap O_j$ soit un domaine fermé de~$O_j$, et leur réunion
contient $W\cap O_j\cap \Upsilon$.
La négligeabilité paralinéaire est une propriété locale;
il suffit de démontrer que pour tout~$j$, $\partial(W)\cap\Upsilon\cap O_j$
est négligeable.
On a 
\[ \partial(W)\cap O_j\cap \Upsilon =\partial(W\cap O_j) \cap \Upsilon
 \subset \bigcup_{i\in I_j} \partial (W\cap Z_i\cap O_j)\cap\Upsilon 
\subset \bigcup_{i\in I_j} \partial(W\cap Z_i)\cap \Upsilon,  \]
donc est négligeable.

Le domaine~$W$ est G-recouvert par des domaines élémentaires compacts
de~$X\times Y$.  D'après ce qui précède,
on peut donc se ramener au cas où $X$ et~$Y$
sont compacts et où il existe une carte G-tropicale~$f$ de domaine~$X$,
une carte G-tropicale~$g$ de domaine~$Y$ telles que
$W=(f\times g)^{-1}(R)$, 
où $R$ est la partie paralinéaire compacte~$(f\times g)(W)$ 

La partie squelettique~$\Upsilon$ est G-recouverte
par des parties paralinéaires de parties squelettiques élémentaires,
de la forme $\Sigma\times\Tau \cap \Omega$, où $\Omega$ est un domaine de~$X\times Y$ et $\Sigma$ et~$\Tau$ sont des parties squelettiques
élémentaires de domaines~$U$ et~$V$ de~$X$ et~$U$.
D'après ce qui précède, il suffit de démontrer que
\[ \partial(W\cap U\times V)\cap \Sigma\times\Tau \]
est négligeable. On peut donc supposer que $U=X$, $V=Y$
et $\Upsilon=\Sigma\times\Tau$.

Comme $\Sigma$ est une partie paralinéaire,
elle est G-recouverte par des parties fidèlement paralinéaires
de la forme $\Sigma\cap U$, où $U$ est un domaine de~$X$.
De même, $\Tau$
est G-recouverte par des parties fidèlement paralinéiares
de la forme $\Tau\cap V$, où $V$ est un domaine de~$Y$.
Cela permet de supposer qu'il existe des cartes tropicales~$f'$
et~$g'$ de domaines~$X$ et~$Y$ respectivement,
fidèles sur~$\Sigma$ et~$\Tau$ respectivement.

On remplace alors $f$ et~$f'$ par $(f,f')$, 
$g$ et~$g'$ par $(g,g')$, de sorte que le couple $(f,g)$ est fidèle
sur $\Sigma\times\Tau$.

La trace de $\partial(X)\times Y$ sur $\Sigma\times\Tau$
est contenue dans $(\partial(X)\cap\Sigma)\times\Tau$
donc est négligeable puisque $\partial(X)\cap\Sigma$ est $m$-négligeable.
De même, 
la trace de $X\times \partial(Y)$ sur $\Sigma\times\Tau$
est négligeable.
Il reste à démontrer que la trace de $\partial(W/(X\times Y))$
sur~$\Sigma\times\Tau$ est négligeable.

Choisissons une décomposition cellulaire~$\mathscr C$
de $f(X)\times g(Y)$ compatible avec~$R$ et avec 
$f(\Sigma)\times g(\Tau)$. Soit~$\mathscr D$ la décomposition 
cellulaire de $\Sigma \times \Tau$ déduite de 
$\mathscr C|_{f(\Sigma)\times g(\Tau)}$ via l'isomorphisme
paralinéaire $(f\times  g)|_{\Sigma\times \Tau}$
de $\Sigma\times \Tau$ sur $f(\Sigma)\times g(\Tau)$. 
Puisque $\mathscr C$ est compatible avec~$R$ 
et puisque $W=(f\times g)^{-1}(R)$, 
la décomposition~$\mathscr D$ est compatible avec $W\cap (\Sigma \times \Tau)$. 
Soit~$D$ une $(m+n)$-cellule de~$\mathscr D$ contenue dans~$W$, 
et posons $C=(f\times g)(D)$. 
Par construction, $C$ est une $(m+n)$-cellule de $\mathscr C$, et $C\subset R$. 
Puisque $f(X)\times g(Y)$ est de dimension au plus $m+n$, 
la cellule~$C$ est maximale dans~$\mathscr C$, 
et $\mathring C$ est dès lors un ouvert de $f(X)\times g(Y)$. 
À ce titre, 
$\mathring C$ est contenu dans $\Int(R/f(X)\times g(Y))$, 
si bien que 
$(f\times g)^{-1}(\mathring C)$ est contenu dans $\Int(W/X\times Y)$. 
En particulier, $\mathring D \subset \Int(W/X\times Y)$
(on a en effet par construction $\mathring D=(f\times g)^{-1}(\mathring C)
\cap (\Sigma\times \Tau)$). 
Ceci étant valable pour toute $(m+n)$-cellule $D$ 
de  $\mathscr D$ qui est contenue dans $W\cap (\Sigma \times \Tau)$, 
il s'ensuit que $\partial (W/(X\times Y))
\cap (\Sigma \times \Tau)$ est contenu dans une réunion de cellules
de~$\mathscr D$ de dimensions strictement inférieures à $m+n$, 
ce qui achève la démonstration.

\section{Formes sur un espace tropical produit}

\begin{prop}\label{prop.decoupage-produit}
Soit $X$ et $Y$ des espaces tropicaux de niveaux~$m$ et~$n$ respectivement.
Soit $\mathscr D_X$ un jeu de cartes tropicales couvrant sur~$X$,
soit $\mathscr D_Y$ un jeu de cartes tropicales couvrant sur~$Y$
et soit $\mathscr D_{X\times Y}$ un jeu de cartes tropicales
couvrant sur~$X\times Y$.

Soit $K,L,M$ des parties compactes de~$X$, $Y$ et~$X\times Y$
respectivement.

Il existe un jeu de cartes tropicales fini $\mathscr E_X$ sur~$X$,
un jeu de cartes tropicales fini~$\mathscr E_Y$ sur~$Y$
tels que 
\begin{enumerate}\def\theenumi{\roman{enumi}}\def\labelenumi{(\theenumi)}
\item Le jeu $\mathscr E_X$ sur~$X$ raffine le jeu~$\mathscr D_X$
et couvre~$K$;
\item Le jeu $\mathscr E_Y$  sur~$Y$ raffine le jeu~$\mathscr D_Y$ 
et couvre~$L$;
\item Le jeu produit $\mathscr E_{X} \times \mathscr E_Y$ 
sur~$X\times Y$ raffine le jeu~$\mathscr D_{X\times Y}$ et couvre~$M$.
\end{enumerate}
\end{prop}
\begin{proof}
Il n'est pas restrictif de supposer que $M=K\times L$
et qu'il n'est pas vide.

Pour tout tout $z\in K\times L$
il existe une carte tropicale produit $(f_z,g_z)$ sur~$X\times Y$,
définie sur un voisinage $U_z\times V_z$ de~$z$, 
qui raffine une carte de $\mathscr D_{X\times Y}$. 
Il n'est pas restrictif de supposer que pour tout~$z$,
la carte $(U_z,f_z)$ raffine~$\mathscr D_X$
et que la carte $(V_z,g_z)$ raffine~$\mathscr D_Y$.
Recouvrons $K\times L$ par une sous-famille de ces domaines de cartes,
indexée par une partie finie~$Z$ de~$K\times L$.
Les cartes $(U_z,f_z)$ ainsi obtenues forment 
un jeu de cartes qui couvre~$K$ et qui raffine le jeu~$\mathscr D_X$.
Les cartes tropicales $(V_z,g_z)$ ainsi obtenues forment 
un jeu de cartes sur~$Y$ qui couvre~$L$ et raffine~$\mathscr D_Y$.

Soit $y\in L$.
L'intersection des~$V_z$, pour $z\in Z$ tels que $y\in V_z$,
est un voisinage~$V'_y$ de~$y$. 
On le munit de la carte
tropicale~$g'_y$ obtenue en concaténant les~$g_z|_{V'_y}$,
pour $z\in Z$ tels que $y\in V_z$.
Chacune de ces cartes raffine le jeu~$\mathscr D_Y$:
soit en effet $y\in Y$; comme les~$V_z$ recouvrent $L$,
il existe $z$ tel que $y\in V_z$ et l'on a alors $V'_y\subset V_z$.
Soit $J$ une partie finie de~$L$ telle que les $V'_y$, pour $y\in J$,
recouvrent la partie compacte~$L$.
Le jeu de cartes tropicales formé des $(V'_y,g'_y)$, pour $y\in J$,
raffine~$\mathscr D_Y$ et couvre~$L$.

Soit $I$ l'ensemble des applications~$i \colon J\to Z$
telles que $y\in V_{i(y)}$.
Pour $i\in I$, soit $U'_i$ l'intersection des~$U_{i(y)}$, pour $y\in J$.
C'est une partie ouverte de~$X$
qu'on munit de la carte tropicale déduite par concaténation des~$f_{i(y)}$,
pour $y\in J$. Comme $J$ n'est pas vide (car $L$ n'est pas vide),
on a l'inclusion $U'_i\subset U_{i(y)}$, où $y\in J$. Par suite, 
cette carte raffine le jeu~$\mathscr D_X$.

Démontrons que le jeu de cartes $(U'_i)_{i\in I}$ couvre~$K$.
Soit en effet $x\in K$. 
Pour tout $y\in J$, il existe $z\in Z$ tel que le point $(x,y)$ 
appartienne à~$U_z\times V_z$; on note $i(y)$ un tel point.
On a ainsi construit un élément $i\in I$ tel que $x\in U'_i$
puisque $x\in U_{i(y)}$ pour tout $y\in J$.
 
Vérifions maintenant que le jeu produit $(U'_i\times V'_y)_{i,y}$
raffine le jeu~$\mathscr D_{X\times Y}$.
Soit $i\in I$ et $y\in J$. Par construction de~$V'_y$,
on a $V'_y\subset V_z$ si $y\in V_z$;  comme $y\in V_{i(y)}$,
on donc $V'_y\subset V_{i(y)}$. Par suite,
on a $U'_i\times V'_y\subset U_{i(y)}\times V'_y\subset U_{i(y)}\times V_{i(y)}$, d'où l'assertion. Il est clair que ce jeu couvre $K\times L$.
\end{proof}

\subsection{}\label{ss.formes-produit-jeu}
Soit $\mathscr D$ un jeu de cartes tropicales couvrant sur $X\times Y$
et soit $M$ une partie compacte de~$X\times Y$.
D'après la proposition~\ref{prop.decoupage-produit},
il existe un jeu de cartes tropicales~$\mathscr E$ sur~$X$ qui couvre~$p_1(M)$
et un jeu de cartes tropicales~$\mathscr E'$ sur~$X$ qui couvre~$p_2(M)$
tels que le jeu produit $\mathscr E\times \mathscr E'$
raffine le jeu~$\mathscr D$ et couvre~$M$. Soit $\mathscr E$
et $\mathscr E'$ de tels jeux.

Soit $(\rho_\lambda)$ une partition de l'unité au voisinage de~$p_1(M)$
sur~$X$ subordonnée aux cartes de~$\mathscr E$
et $(\rho'_\mu)$ une partition de l'unité au voisinage de~$p_2(M)$
sur~$Y$ subordonnée aux cartes de~$\mathscr E'$.
Alors la famille $(\rho_\lambda\rho'_\mu)$ est une partition de l'unité
au voisinage de~$M$, subordonnée aux cartes de $\mathscr E\times\mathscr E'$.
Pour toute forme $\omega\in\mathscr A^{*}_{\mathscr D,M}$,
on peut écrire
\begin{equation}\label{eq.dec-omega}
 \omega = \sum_{\lambda,\mu} \rho_\lambda \rho'_\mu \omega.
\end{equation}
Pour tout couple $(\lambda,\mu)$, la forme $\omega|_{U_\lambda\times V_\mu}$
sur $U_\lambda\times V_\mu$ est tropicalisée par une carte $(f_\lambda,g_\mu)$.
Choisisssons un domaine compact~$W_{\lambda,\mu}$ de $U_\lambda\times V_\mu$
qui contient un voisinage de $\supp(\rho_\lambda)\times \supp(\rho'_\mu)$.
Il découle de l'exemple~\ref{exem.forme-pl-produit}
que $\omega|_{W_\lambda,\mu}$ est de la forme
\[ \sum_i 
(f_\lambda,g_\mu)^* \gamma_{\lambda,\mu,i} 
p_1^*f_\lambda^*\alpha_{\lambda,\mu,i}\wedge p_2^*g_\mu^*\beta_{\lambda,\mu,i},
\]
où les $\alpha_{\lambda,\mu,i}$ sont des formes de type~$(p,q)$ 
sur l'espace affine but de~$f_\lambda$
et les $\beta_{\lambda,\mu,i}$ sont des formes de type~$(p',q')$ 
sur l'espace affine but de~$g_\mu$
et les $\gamma_{\lambda,\mu,i}$ sont des fonctions lisses
sur le produit de ces deux espaces affines.

Pour tout triplet $(\lambda,\mu,i)$,  le support de la forme
\[ \rho_\lambda \rho'_\mu (f_\lambda,g_\mu)^*\gamma_{\lambda,\mu,i} p_1^*f_\lambda^* \alpha_{\lambda,\mu,i} \wedge p_2^* g_\mu^* \beta_{\lambda,\mu,i} \]
sur $U_\lambda\times V_\mu$ est contenu dans le
produit $\mathring U_\lambda\times \mathring V_\mu$.
On note $j_!(\cdot)$ son prolongement par zéro.
De même, le support de la forme $\rho_\lambda f_\lambda^*\alpha_{\lambda,\mu,i}$
sur~$U_\lambda$ est contenu dans~$\mathring U_\lambda$;
on note $\widetilde\alpha_{\lambda,\mu,i}$ son prolongement par zéro.
On définit de même $\widetilde\beta_{\lambda,\mu,i}$,
le prolongement par zéro de $\rho'_\mu g_\mu^*\beta_{\lambda,\mu,i}$.

On a ainsi
\begin{align*}
\omega & = \sum_{\lambda,\mu} \rho_\lambda \rho'_\mu \omega  \\
& = \sum_{\lambda,\mu} j_! (\rho_\lambda \rho'_\mu \omega|_{W_{\lambda,\mu}}) \\
& = \sum_{\lambda,\mu,i} j_! (\rho_\lambda \rho'_\mu (f_\lambda,g_\mu)^*\gamma_{\lambda,\mu,i} p_1^*f_\lambda^* \alpha_{\lambda,\mu,i} \wedge p_2^* g_\mu^* \beta_{\lambda,\mu,i})  \\
& = \sum_{\lambda,\mu,i} j_! ((f_\lambda,g_\mu)^*\gamma_{\lambda,\mu,i}
p_1^* \widetilde \alpha_{\lambda,\mu,i} \wedge p_2^*\widetilde\beta_{\lambda,\mu,i}).
\end{align*} 
Pour tout triplet $(\lambda,\mu,i)$ comme ci-dessus,
on peut choisir une fonction lisse~$\theta_{\lambda,\mu,i}$
sur~$X\times Y$
qui vaut~$1$ au voisinage de $\supp(\rho_\lambda)\times\supp(\rho'_\mu)$
et dont le support est contenu dans l'intérieur de $U_\lambda\times V_\mu$.
On a donc
\begin{align*}
\omega 
& = \sum_{\lambda,\mu,i} j_! (\theta_{\lambda,\mu,i} (f_\lambda,g_\mu)^*\gamma_{\lambda,\mu,i}
p_1^* \widetilde \alpha_{\lambda,\mu,i} \wedge p_2^*\widetilde\beta_{\lambda,\mu,i}) \\
& = \sum_{\lambda,\mu,i} j_! (\theta_{\lambda,\mu,i} (f_\lambda,g_\mu)^*\gamma_{\lambda,\mu,i}) p_1^* \widetilde \alpha_{\lambda,\mu,i} \wedge p_2^*\widetilde\beta_{\lambda,\mu,i}.
\end{align*}

En particulier, on a le corollaire suivant.

\begin{coro}\label{coro.formes-produit}
Toute forme à support compact $\omega$ sur~$X\times Y$
est une somme finie de formes à support compact du type 
\[ \gamma p_1^*\alpha \wedge p_2^*\beta \]
où $\alpha$ est une forme à support compact sur~$X$,
$\beta$ est une forme à support compact sur~$Y$,
et $\gamma$ est une fonction lisse à support compact sur~$X\times Y$.
\end{coro}

\begin{coro}\label{coro.densite-formes-produit}
Les formes sur $X\times Y$ du type $p_1^*\alpha\wedge p_2^*\beta$,
où $\alpha$ et $\beta$ sont des formes lisses à support compact sur~$X$
et~$Y$ respectivement,
engendrent un sous-espace vectoriel dense des formes lisses à support
compact sur~$X\times Y$.
\end{coro}
\begin{proof}
Compte tenu du corollaire précédent, il suffit de vérifier
que toute \emph{fonction} lisse à support compact~$c$ sur~$X\times Y$
est approchable par des fonctions de la forme $p_1^* a \cdot p_2^* b$,
où $a$ et $b$ sont des fonctions lisses à support compact sur~$X$
et~$Y$ respectivement.
Soit $K$ et $L$ des parties compactes de~$X$ et~$Y$ telles que $K\times L$
contienne le support de~$c$.
D'après la proposition~\ref{prop.decoupage-produit},
il existe  des jeux de cartes tropicales
finis, $\mathscr E_X$ sur~$X$, et $\mathscr E_Y$ sur~$Y$,
qui couvrent~$K$ et~$L$ respectivement,
telles que le jeu produit $\mathscr E_X\times \mathscr E_Y$
tropicalise~$c$ au voisinage de~$K\times L$.
Soit $(\lambda_i)$ une partition de l'unité sur~$X$ au voisinage de~$K$
subordonnée au jeu~$\mathscr E_X$ et
soit $(\mu_j)$ une partition de l'unité sur~$Y$ au voisinage de~$L$
subordonnée au jeu~$\mathscr E_Y$.
On a alors 
$ c = \sum_{i,j} \lambda_i \mu_j c$.
Fixons $i$ et~$j$. Soit $U$ le domaine d'une carte~$f\colon U\to\R^m$ 
de~$\mathscr E_X$
telle que $\supp(\lambda_i)\subset U$ et soit $V$ le domaine d'une carte~$g\colon V\to\R^n$
de~$\mathscr E_Y$ telle que $\supp(\mu_j)\subset V$.
Sur $U\times V$, la fonction~$c$ est tropicalisée par la carte $(f,g)$
et il existe une fonction lisse~$\gamma$ sur~$\R^{m+n}$
telle que $c=(f,g)^*\gamma$. Il existe alors une suite $(P_n)$
de polynômes sur~$\R^{m+n}$ 
telle que $(\partial^kP_n)$ converge uniformément vers~$\partial^k\gamma$
sur toute partie compacte de~$\R^{m+n}$,
pour tout multi-exposant $k\in\N^{m+n}$. 
Alors $(f,g)^*P_n$ converge vers~$c$ 
dans $\mathscr A^0_{\mathscr E_X\times\mathscr E_Y}(X\times Y)$.
Par suite, 
$\lambda_i\mu_j(f,g)^*P_n$ converge vers $\lambda_i\mu_j c$
dans $\mathscr A^0_\cpct(X\times Y)$.
Comme ce sont des fonctions de la forme requise, le corollaire s'en déduit.
\end{proof}

\subsection{}\label{ss.formes-produit-jeu'}
Reprenons les notations du paragraphe~\ref{ss.formes-produit-jeu}
en supposant de plus que $Y$ est compact et que le jeu~$\mathscr E'$
sur~$Y$ est réduit à une seule carte~$g$ de domaine~$Y$.
La partition de l'unité $(\rho'_\mu)$ est réduite à la fonction constante~$1$,
et on peut supposer que les fonctions $\theta_{\lambda,\mu,i}$ 
se déduisent de fonctions lisses sur~$X$.
La forme~$\omega$ considérée se décompose alors comme
\[ \omega = \sum_{\lambda, i} j_!(\theta_{\lambda,i} (f_\lambda,g)^*\gamma_{\lambda,i}) p_1^*\widetilde \alpha_{\lambda,i} \wedge p_2^* g^*\beta_{\lambda,i}.\]

\subsection{}
On note $\mathscr A^{p',q'}_X\boxtimes \mathscr A^{p'',q''}_{Y}$
le sous-faisceau de $\mathscr A^{\cdot,\cdot}_{X\times Y}$
en $\mathscr A_{X\times Y}$-modules
engendré par les formes du type $p_1^*\alpha\wedge p_2^*\beta$
où $\alpha \in \mathscr A^{p',q'}_{X}$
et $\beta \in \mathscr A^{p'',q''}_{Y}$;
il est contenu dans $\mathscr A^{p'+p'',q'+q''}_{X\times Y}$.
Ses sections sont appelées formes de type $(p',q';p'',q'')$
sur $X\times Y$.  Lorsque $p''=q''=0$, on dira plus simplement
$(p',q';0)$-forme.

%

\begin{prop}\label{prop.dec-forme-produit}
\begin{enumerate}
\item
On a 
\[ 
 \mathscr A^{p,q}_{X\times Y} = \bigoplus
\mathscr A^{p',q'}_{X}\boxtimes \mathscr A^{p'',q''}_{Y}, \]
où la somme directe porte sur les $(p',q',p'',q'')\in\N^4$ 
tels que $p=p'+p''$ et $q=q'+q''$.

 
\item
Soit $(f\colon X\to E,P)$ et $(g\colon Y\to F, Q)$ des cartes tropicales, 
et soit $\omega$ une forme sur~$X\times Y$
qui est tropicalisée par 
la carte tropicale~$(f,g)\colon X\times Y\to E\times F$
d'image contenue dans $P\times Q$.
Pour tout $(p',q',p'',q'')$, la composante de~$\omega$ sur 
$\mathscr A^{p',q'}_{X_\groth}\boxtimes \mathscr A^{p'',q''}_{Y_\groth}$
est somme de termes du type
$\rho\cdot (f\circ p_1)^*\alpha\wedge (g\circ p_2)^*\beta$, 
où $\rho$ est une fonction lisse sur $X\times Y$ tropicalisée par~$(f,g)$,
$\alpha$ est une $(p',q')$-forme sur~$P$
et $\beta$ est une $(p'',q'')$-forme sur~$Q$. 
\end{enumerate}
\end{prop}
\begin{proof}
La démonstration se fait en plusieurs étapes.

1) Démontrons d'abord une assertion d'unicité.
Soit $W$ un ouvert de~$X\times Y$ et, 
pour tout $\pi=(p',q',p'',q'')$, soit $\omega_\pi$
une forme sur~$W$ appartenant à 
$\mathscr A^{p',q'}_{X_\groth}\boxtimes \mathscr A^{p'',q''}_{Y_\groth}$;
supposons que $\sum \omega_\pi=0$ et démontrons que $\omega_\pi=0$
pour tout~$\pi$.

Soit $(x,y)\in W$; démontrons que pour tout~$\pi$,
$\omega_\pi$ est nulle au voisinage de~$(x,y)$. 
Il existe un voisinage domanial compact~$U$ de~$x$
et un voisinage domanial compact~$V$ de~$y$,
une carte tropicale~$f\colon U\to E$ 
et une carte tropicale~$g\colon V\to F$
telles que  pour tout~$\pi$, $\omega_\pi$ soit tropicalisée
par la carte tropicale $(f,g)$ sur~$U\times V$.
Il suffit de prouver que pour tout~$\pi$,
la restriction de~$\omega_\pi$ à $U\times V$ est nulle.
Pour tout~$\pi$, soit $\gamma_\pi$ une forme sur $f(U)\times g(V)$
telle que $\omega_\pi=(f,g)^*\gamma_\pi$; 
compte tenu du lemme~\ref{lemm.extensions},
on peut même supposer
que $\gamma_\pi$ est la restriction d'une forme sur~$E\times F$.

La nullité de $\omega=\sum \omega_{\pi}$
équivaut à la nullité de la forme $\gamma=\sum \gamma_{\pi}$
sur le sous-espace paralinéaire compact $f(U)\times g(V)$,
et de même pour chacune de ses composantes $\omega_{\pi}$.
Soit $\mathscr C$ une décomposition cellulaire de $f(U)$,
$\mathscr D$ une décomposition cellulaire de~$g(V)$,
et considérons la décomposition cellulaire produit de $f(U)\times g(V)$.
Soit $C\in\mathscr C$ et $D\in\mathscr D$;
par hypothèse $\gamma|_{C\times D}=0$, et il s'agit
de démontrer que $\gamma_{\pi}|_{C\times D}=0$.
Cela ramène à prouver l'assertion sur une cellule-produit,
puis au cas de $\R^d\times\R^e$,
auquel cas elle est évidente par un calcul explicite.

2) Prouvons un énoncé d'existence.
Soit $\omega$ une forme sur $X\times Y$,
soit $(f\colon U\to \R^d,P)$ 
une carte tropicale définie sur un ouvert~$U$ de~$X$,
soit $(g\colon V\to \R^e,Q)$ une carte tropicale définie sur un ouvert~$V$ de~$Y$
et supposons que la restriction de~$\omega$ à l'ouvert $U\times V$
de~$X\times Y$ soit tropicalisée par la carte tropicale~$(f,g)$
d'image contenue dans la partie paralinéaire~$P\times Q$.
Soit $\Omega$ un voisinage ouvert de~$P\times Q$ dans $\R^d\times \R^e$
et soit $\gamma$ une forme sur~$\Omega$
telle que $\omega|_{U\times V}= (f,g)^*\gamma$.
D'après l'exemple~\ref{exem.forme-pl-produit},
on peut décomposer la forme $\gamma$ en  une somme
$\sum_\pi \sum_{I,J,K,L} \rho^{IJKL}_\pi \di x_I \wedge \dc x_J \wedge \di y_K \wedge \dc y_L$, 
où $I,J,K,L$ sont des multindices de longueurs $p',p'',q',q''$,
avec $\pi=(p',q',p'',q'')$, 
les variables de~$\R^d$ étant notées~$x_i$ et celles de~$\R^e$ étant
notées~$y_j$,
et où les $\rho^{IJKL}_\pi$ sont des fonctions lisses sur~$\Omega$.
Ainsi, 
\[ (f,g)^*\gamma = 
\sum_\pi \sum_{I,J,K,L} (f,g)^*\rho^{IJKL}_\pi f^*(\di x_I \wedge \dc x_J) \wedge g^*(\di y_K \wedge \dc y_L), \] 
qui est une décomposition du type demandé.

3) Démontrons l'assertion~\emph a) de la proposition.
L'indépendance des différents facteurs résulte de l'étape~1).
Démontrons que leur somme est le faisceau~$\mathscr A^{\cdot,\cdot}_{X\times Y}$.
Soit $\omega$ une forme sur un ouvert~$W$ de~$X\times Y$.
L'étape~2) démontre que $\omega$ appartient à localement à la somme considérée;
elle y appartient donc globalement.

4) L'assertion~\emph b) est alors le cas particulier de l'étape~2)
appliquée à $U=X$ et $V=Y$.
\end{proof}

\subsection{}
Soit $X$ et $Y$ des espaces tropicaux de niveaux~$m$ et~$n$ respectivement.
L'opérateur différentiel $\di$ sur 
sur les formes sur~$X\times Y$
se décompose comme la somme de deux opérateurs différentiels~$\di[X]$
et~$\di[Y]$.
Si $\omega$ est une forme de type $(p',q';p'',q'')$, 
on note $\di[X]\omega$ et $\di[Y]\omega$ les composantes de $\di\omega$
de type $(p'+1,q';p'',q'')$ et $(p',q';p''+1,q'')$ respectivement.
On vérifie dans des cartes que l'on a $\di\omega=\di[X]\omega+\di[Y]\omega$.
Ces applications s'étendent par linéarité à toutes les formes.

Par construction, si $\alpha$ est une forme sur~$X$, on a
\[ \di[X] (p_1^*\alpha)=p_1^*(\di \alpha) \quad\text{et}\quad \di[Y] (p_1^*\alpha)=0, \]
et des formules analogues si $\beta$ est une forme sur~$Y$.

De la formule de Leibniz 
\[ \di (\omega\wedge\eta)= \di\omega \wedge \eta+ (-1)^{p+q}\omega\wedge \di\eta \]
si $\omega$ est de type $(p,q)$, on déduit 
les deux égalités
\[ \di[X] (\omega\wedge\eta)= \di[X]\omega \wedge \eta+ (-1)^{p+q}\omega\wedge \di[X]\eta \]
et
\[ \di[Y] (\omega\wedge\eta)= \di[Y]\omega \wedge \eta+ (-1)^{p+q}\omega\wedge \di[Y]\eta .\]

On définit de même les opérateurs~$\dc[X]$ et~$\dc[Y]$.

\subsection{}
Par transposition, on déduit des opérateurs~$\di[X]$, $\dc[X]$, $\di[Y]$ et~$\dc[Y]$
des opérateurs linéaires analogues sur les espaces de courants, 
avec les mêmes règles de signes que celles du paragraphe~\ref{ss.calcul-diff-courants}. Par exemple, si $S$ est un courant et~$\omega$ une   forme
de type~$(p,q)$ sur~$X\times Y$, on a 
\[ \langle \di[X] S, \omega\rangle=(-1)^{p+q+1} \langle S, \di[X]\omega\rangle.\]

\section{Intégrales dépendant de paramètres}

\begin{lemm}\label{lemm.continuite-fibre}
Soit $X$ et $Y$ des espaces tropicaux de niveaux~$m$ et~$n$ respectivement.
Soit $r$ un élément de $\N\cup\{+\infty\}$.

Soit $\Phi$ la famille des parties fermées de~$X\times Y$
qui sont compactes au-dessus de~$Y$.

Soit $\omega\in\mathscr A_\Phi(X\times Y)$.
Pour tout $y\in Y$, on note $\omega_y$ la forme sur~$X$,
image réciproque de~$\omega$ par l'application $x\mapsto (x,y)$.

\begin{enumerate}
\item
Pour tout $y\in Y$, le support de~$\omega_y$ est compact.

\item
L'application $y\mapsto \omega_y$ de~$Y$
dans~$\mathscr A_{\cpct}(X)$ est continue lorsque ce dernier espace
est muni de la topologie~$\mathscr C^r$.
\end{enumerate}
\end{lemm}
\begin{proof}
Il n'y a rien à démontrer si $Y$ est vide; nous le supposons désormais non vide.

\begin{enumerate}
\item
Le support de~$\omega|_{X\times\{y\}}$ est contenu 
dans $\supp(\omega)\cap X\times\{y\}$. Il est donc compact.

\item
L'assertion est locale sur~$Y$, ce qui permet de 
se ramener au cas où $Y$ est compact. 
La famille~$\Phi$ est alors la famille des parties compactes de~$X\times Y$.
 
Soit $\mathscr D$ un jeu de cartes couvrant sur~$X\times Y$ 
qui tropicalise~$\omega$.
Soit $K$ la projection du support de~$\omega$ sur~$X$;
c'est une partie compacte de~$X$.
Appliquons la proposition~\ref{prop.decoupage-produit}
à la partie compacte $K\times Y$ de~$X\times Y$ :
il existe ainsi
des jeux de cartes finis~$\mathscr E$ sur~$X$
et~$\mathscr E'$ sur~$Y$ tels que 
le jeu produit $\mathscr E\times\mathscr E'$ raffine~$\mathscr D$
et couvre $K\times Y$.

Pour tout $y\in Y$, la forme $\omega_y$ est tropicalisée par~$\mathscr E$
et son support est contenu dans~$K$.
L'application $y\mapsto \omega_y$ est donc à valeurs
dans $\mathscr A_{\mathscr E,K}(X)$.
Pour démontrer qu'elle est continue pour la topologie~$\mathscr C^r$,
il suffit donc de prouver que c'est le cas
de l'application induite de~$Y$ dans~$\mathscr A_{\mathscr E,K}(X)$.

Soit $W$ un domaine compact contenue dans une des cartes~$(f\colon U\to E)$ 
du jeu~$\mathscr E$ et soit $\norm\cdot_{W,r}$ 
une seminorme~$\mathscr C^r$ sur $\mathscr A(X)$ 
associée à une décomposition cellulaire de~$f(W)$ et
un système de coordonnées sur ses cellules.
Il s'agit de prouver que pour $z$ assez proche de~$y$,
$\norm{\omega_z-\omega_y}_{W,r}$ est aussi petit que l'on veut.
Soit $(g\colon V\to F)$ une carte du jeu~$\mathscr E'$ tel
que $y$ appartient à l'intérieur de~$V$; 
soit $W'$ un voisinage domanial compact de~$y$ contenu dans~$V$.
Par définition, il existe une forme $\alpha$
sur $f(W)\times g(W')$ telle que 
$\omega|_{W\times W'}=(f,g)^*\alpha$; on peut supposer que $\alpha$
est une forme lisse sur $E\times F$.
Pour $z\in W'$, la forme $\omega_z|_W$ est 
égale à $f^*\alpha|_{g(z)}$, où pour tout $t\in F$, $\alpha_t$
est la restriction de~$\alpha$ à~$E\times\{t\}$, identifié à~$E$.

Posons $P=f(W)$ et $Q=g(W')$. 
Par définition, $\norm{\omega_z}_{W,r}=\norm{\alpha_{g(z)}}_{P,r}$.
On se ramène ainsi la continuité de l'application $t\mapsto \alpha_t$
de~$Q$ dans~$\mathscr A(P)$ pour la norme~$\norm\cdot_{P,r}$.

En raisonnant cellule par cellule, on se ramène au cas où $P$
est une cellule~$C$. Alors, les formes sur~$C$ ont une base
standard comme  module sur sur l'espace des fonctions lisses sur~$C$.
En fixant une telle base, on se ramène au cas 
où $\alpha$ est une fonction lisse.
Il s'agit alors de prouver que lorsque $\phi$ est une fonction
lisse sur $E\times F$, l'application $t\mapsto\phi(\cdot,t)$
de $F$ dans~$\mathscr C^\infty(P)$  
est continue pour la topologie~$\mathscr C^r$,
ce qui résulte de la formule de Taylor.
\qedhere
\end{enumerate}
\end{proof}

\subsection{}
Soit $X$ et $Y$ des espaces tropicaux de niveaux~$m$ et~$n$ respectivement.
Soit $S$ un courant de bidimension~$(p,q)$ sur~$X$.
Soit $\gamma$ une $(p,q;0)$-forme lisse sur~$X\times Y$ dont le support
est compact au-dessus de~$Y$.

Pour tout $y\in Y$,  on peut restreindre $\gamma$ à $X\times \{y\}$;
on obtient ainsi une $(p,q)$-forme lisse dont le support est compact,
qu'on notera $\gamma(\cdot,y)$. On note alors 
$S_*(\gamma)(y)$ la valeur du courant~$S$ sur cette forme.
On a ainsi défini une fonction $S_*(\gamma)$ sur~$Y$. 

\begin{lemm}\label{lemm.support-S*}
Le support de $S_*(\gamma)$ est contenue dans la partie 
fermée $p_2(\supp(\gamma))$.
\end{lemm}
\begin{proof}
Comme $\supp(\gamma)$ est compact au-dessus de~$Y$, son image
est fermée dans~$Y$.

Supposons que $y$ n'appartienne pas à l'image du support de~$\gamma$.
Par compacité de ce support au-dessus de~$Y$, 
il existe un voisinage ouvert~$U$ de~$y$ dans~$Y$ tel que $\gamma|_{X\times U}=0$.
Alors $\gamma(\cdot, y')=0$ pour tout $y'\in U$
si bien que $S_*(\gamma)$ est nulle sur~$U$. En particulier,
$y$ n'appartient au support de~$S_*(\gamma)$.
\end{proof}

\begin{prop}\label{prop.lisse-parametre}
Soit $f$ une carte tropicale sur un domaine~$U$ de~$X$, 
soit $g$ une carte tropicale sur un domaine~$V$ de~$Y$,
soit $\rho$ une fonction lisse sur~$X$ dont
le support est compact et contenu dans l'intérieur de~$U$.

\begin{enumerate}
\item
Soit $\gamma$ une forme lisse sur de bidegré~$(p,q;0)$ sur~$X\times Y$
qui est tropicalisée par $(f,g)$ sur $U\times V$.
La forme $\rho\gamma$ sur~$X\times Y$ est à support compact
au-dessus de~$Y$, de type~$(p,q;0)$, 
et la fonction $S_*(\rho\gamma)$ 
est lisse sur~$V$, tropicalisée par~$g$.

\item
L'application $\mathscr A^{p,q;0}_{(f,g)}(X\times Y) \to \mathscr A^{0}_g(Y)$
donnée par $\gamma\mapsto S_*(\rho\gamma)$ est continue.
\end{enumerate}
\end{prop}
\begin{proof}
\begin{enumerate}
\item
Notons $E$ et $F$ les espaces affines buts de~$f$ et~$g$.
Il s'agit de prouver que la restriction
de~$S_*(\rho\gamma)$ à tout domaine compact~$W$ de~$V$
est lisse et tropicalisée par~$g$.

Soit $W$ un tel domaine et
soit $K$ un domaine compact qui contient le support de~$\rho$.
Considérons un jeu de cartes tropicales fini sur~$X$ qui couvre~$K$,
toutes plus fines que~$f$, et qui tropicalise~$\rho$,
et adjoignons-lui la carte tropicale triviale de domaine
le complémentaire de~$K$. 
Notons~$\mathscr E$ le jeu ainsi obtenu.

Par définition d'une forme lisse tropicalisée,
il existe une forme lisse~$\alpha$ 
sur la partie paralinéaire~$f(K)\times g(W)$
telle  que $\gamma|_{K\times W}= (f,g)^*\alpha$.
Cette forme se prolonge en une forme lisse sur l'espace
affine $E\times F$, qu'on notera encore~$\alpha$

Pour toute forme lisse~$\beta$ de type~$(p,q;0)$ sur~$E\times F$, 
on définit une fonction $T_\beta$ sur~$F$ de la façon suivante.
Considérons la forme $\rho(x)f^*\beta$ sur $U\times F$;
son support est contenu dans $\supp(\rho)\times F$;
on la prolonge (par zéro) en une forme lisse~$\omega$ sur $X\times F$.
Pour tout~$z$, la forme~$\omega_z$ sur~$X$ est lisse et à support
compact (contenu dans le support de~$\rho)$ et l'on pose
$T_\beta(z)=\langle S,\omega_z\rangle$.

Nous allons
prouver que $T(\beta)$ est $\mathscr C^\infty$ pour toute forme
lisse~$\beta$ sur~$E\times F$.
Comme dans~\cite[chapitre~4]{Schwartz-1966}, 
cela va découler de la définition de la continuité d'un courant 
et des topologies sur les espaces de formes.

Lorsque $z'$ tend vers un point~$z$,
les formes lisses~$\omega_{z'}$ convergent 
vers la forme lisse~$\omega_z$ dans~$\mathscr A^{p,q}_\cpct(X)$,
d'après le lemme~\ref{lemm.continuite-fibre}.
Puisque~$S$ est une forme linéaire continue sur~$\mathscr A^{p,q}_\cpct(X)$,
il en résulte que $T_\beta(z')$ tend vers~$T_\beta(z)$.
Cela prouve que la fonction~$T_\beta$ est continue.

On fait un raisonnement similaire pour les dérivées partielles de~$T_\beta$.
Fixons un vecteur~$e$ de~$F$ et notons~$\partial_e\psi$ la dérivée partielle
de~$\psi$ par rapport à~$e$.
Alors, les formes lisses sur~$X$ (dépendant de~$t$) données par
\[ \frac{ \omega_{z+te}- \omega_z} t \]
convergent vers la forme~$\partial_e\omega_z$ qui est
définie de façon analogue à~$\omega$ 
en remplaçant~$\beta$ par~$\partial_e\beta$.
Par continuité du courant~$S$, on a donc 
\[ \lim_{t\to0} \frac{T_\beta(z+te)-T_\beta(z)}t 
    = \langle S, \partial_e\omega_z\rangle.\]
Cela prouve que $T_\beta$ 
admet une dérivée partielle par rapport au vecteur~$e$.
De plus, cette dérivée partielle est une fonction continue de~$z$, 
car elle est égale à $T_{\partial_e\beta}$.
Ainsi, $T_\beta$ est de classe~$\mathscr C^1$.
En itérant l'opération, 
on en déduit que $T_\beta$ est de classe~$\mathscr C^\infty$.

Pour tout $v\in W$, la forme $\rho(\cdot)\gamma(\cdot,v)$
sur~$X$ coïncide avec $\rho(\cdot) f^*\alpha(\cdot,g(v))$ sur~$U$,
et est nulle en dehors de~$K$. Cela prouve que que
\[ S_*(\rho\gamma) (v) = T_\alpha (g(v)) . \]
Par conséquent, $S_*(\rho\gamma)$ est lisse et tropicalisée
par~$g$.

\item
Démontrons maintenant la continuité de l'application linéaire
de~$\mathscr A^{p,q;0}_{(f,g)}(X\times Y)$ dans~$\mathscr A^0_g(Y)$
donnée par $\gamma\mapsto S_*(\rho\gamma)$.

Par définition de la continuité du courant~$S$, il existe
une famille finie~$(W_i)$ de domaines compacts de~$X$,
chacun étant contenu dans un domaine du jeu de cartes~$\mathscr E$,
un entier~$r$, et un nombre réel~$c$
tels que $\abs{\langle S, \theta\rangle}\leq c \sup_i \norm{\theta}_{W_i,r}$
pour toute forme lisse~$\theta$ sur~$X$
qui appartient à~$\mathscr A^{p,q}_{\mathscr E,K}(X)$.

Pour tout  $v\in W$, 
\[ S_*(\rho (f,g)^*\alpha)(v) = T_\alpha (g(v)) \]
et par construction, 
$T_\alpha (z) =\langle S, \rho (\cdot)f^*\alpha(\cdot, z)\rangle$
pour tout $z\in F$.
Les formes $\rho(\cdot)\alpha(\cdot,z)$ 
appartiennent à $\mathscr A^{p,q}_{\mathscr E,K}(X)$.
On a donc
\[ \abs{T_\alpha(z)} \leq  c \sup_i \norm{\rho(\cdot)f^*\alpha(\cdot)z}_{W_i,r}.\]
Il existe un nombre réel~$c'$ tel que 
$\norm{\rho(\cdot)f^*\alpha(\cdot,z)}_{W_i,r}\leq 
    c' \norm{\alpha}_{f(K)\times g(W_i),r}$
pour toute forme lisse~$\alpha$ sur~$E\times F$.
Ainsi, il existe  un nombre réel~$c''$ tel que
\[ \abs{T_\alpha(z)} \leq  c'' \sup_i \norm{\alpha}_{f(K)\times g(W_i),r} \]
pour toute forme lisse~$\alpha$ sur~$E\times F$ et tout $z\in g(W)$.
Cela entraîne que $\gamma\mapsto S_*(\rho\gamma)$
est continue lorsque la source
est munie de la topologie~$\mathscr C^r$,
et le but est muni de la topologie~$\mathscr C^0$.

Compte tenu des formules $\partial_e T_\alpha=T_{\partial_e \alpha}$,
on en déduit que $\gamma\mapsto S_*(\rho\gamma)$
est continue lorsque la source est munie de la topologie~$\mathscr C^{k+r}$
et le but est muni de la topologie~$\mathscr C^k$, pour tout entier~$k$.

Par conséquent, l'application $\gamma\mapsto S_*(\rho\gamma)$
est continue lorsque la source et le but sont munis
de la topologie~$\mathscr C^\infty$.
\end{enumerate}
\end{proof}

\begin{coro}\label{coro.S*-lisse}
Soit $\mathscr D$ un jeu de cartes sur~$X\times Y$
et soit $M$ une partie de~$X\times Y$ qui est compacte au-dessus de~$Y$.
Soit $\mathscr E_X$ et $\mathscr E_Y$ des jeux de cartes tropicales
sur~$X$ et~$Y$ tel que le jeu de cartes $\mathscr E_X\times \mathscr E_Y$
raffine~$\mathscr D$ et couvre~$M$.

Pour toute forme lisse $\gamma\in \mathscr A^{p,q;0}_{\mathscr D,M}(X\times Y)$
et toute fonction lisse~$\rho$ sur~$X$,
la fonction $S_*(\rho\gamma)$ 
appartient à $\mathscr A^0_{\mathscr E_Y,p_2(M)}(Y)$.
\end{coro}
\begin{proof}
D'après le lemme~\ref{lemm.support-S*},
le support de $S_*(\rho\gamma)$ est contenu dans $p_2(M)$.
Il suffit alors de démontrer que pour toute carte tropicale~$g$
appartenant au jeu~$\mathscr E_Y$,
la fonction~$S_*(\rho\gamma)$ est lisse et tropicalisée par~$g$ sur
le domaine~$V$ de~$g$.
Cela revient à dire que pour tout domaine compact~$W$
de~$V$, la restriction à~$W$ de~$S_*(\rho\gamma)$
est lisse et tropicalisée par~$g$.

La partie $M\cap (X\times W)$ de~$X\times Y$ est compacte;
soit $K$ sa projection sur~$X$. Le jeu~$\mathscr E_X$ couvre~$K$.
Soit $(\rho_\lambda)$ une partition de l'unité
sur~$X$ au voisinage de~$K$ subordonnée aux cartes de~$\mathscr E_X$.
Comme le support de~$\gamma$ est contenu dans~$M$, 
on en déduit que $\gamma$ coïncide avec 
$ \sum_{\lambda} \rho_\lambda \gamma $ sur $X\times W$,
si bien que 
$ S_*(\rho\gamma) $ coïncide avec $ \sum_\lambda S_*(\rho\rho_\lambda\gamma)$
sur~$W$.
Il suffit alors de démontrer que chacun des termes de cette somme est lisse
et tropicalisé par~$g$ sur~$W$.

Fixons un indice~$\lambda$ et considérons une carte~$f$ du jeu~$\mathscr E_X$
telle que le support de~$\rho_\lambda$ soit contenu dans le domaine~$U$ de~$f$.
L'assertion découle alors de la proposition~\ref{prop.lisse-parametre}.
\end{proof}

\begin{theo}\label{theo.S*-continu}
\begin{enumerate}
\item
L'application $\gamma\mapsto S_*(\gamma)$ définit une application
linéaire continue de $\mathscr A^{p,q;0}_\cpct(X\times Y)$ dans
$\mathscr A^0_\cpct(Y)$.

\item
Soit $\Phi$ la famille de supports constituée des parties fermées
de~$X\times Y$ qui sont compactes au-dessus de~$Y$. 
L'application $\gamma\mapsto S_*(\gamma)$ définit une application
linéaire continue de $\mathscr A^{p,q;0}_{\Phi}(X\times Y)$ dans
$\mathscr A^0(Y)$.
\end{enumerate}
\end{theo}
\begin{proof}
\begin{enumerate}
\item
Il suffit de démontrer que  pour tout jeu de cartes tropicales
couvrant~$\mathscr D$ et toute partie compacte~$M$ de~$X\times Y$,
l'application $S_*$ induit une application continue 
de $\mathscr A^{p,q;0}_{\mathscr D,M}(X\times Y)$
dans $\mathscr A^0_\cpct(Y)$. 
Son image est contenue dans $\mathscr A^0_{p_2(M)}(Y)$.

D'après la proposition~\ref{prop.decoupage-produit},
il existe des jeux de cartes tropicales~$\mathscr E_X$ et~$\mathscr E_Y$
sur~$X$ et~$Y$ respectivement, finis, tels que 
le jeu produit $\mathscr E_X\times \mathscr E_Y$ raffine~$\mathscr D$
et couvre~$M$.
D'après le corollaire~\ref{coro.S*-lisse},
on a $S_*(\gamma)\in \mathscr A^0_{\mathscr E_Y,p_2(M)}(Y)$ 
pour tout $\gamma\in\mathscr A^{p,q;0}_{\mathscr D,M}(X\times Y)$.
Il suffit de démontrer que l'application de 
$\mathscr A^{p,q;0}_{\mathscr D,M}(X\times Y)$ dans
$\mathscr A^0_{\mathscr E_Y,p_2(M)}(Y)$ induite par~$S_*$ est continue.

Soit $(\rho_\lambda)$ une partition de l'unité sur~$X$ au voisinage
du compact~$p_1(M)$, subordonnée au recouvrement formé
des intérieurs des domaines des cartes de~$\mathscr E_X$.
Comme dans~\emph a), on a 
\[ S_*(\gamma) = \sum_\lambda S_*(\rho_\lambda \gamma). \]
Il suffit de démontrer que chacune des applications linéaires
$\gamma \mapsto S_*(\rho_\lambda\gamma)$ est continue.

Soit $\lambda$ un indice, soit $f$ la carte correspondante de~$\mathscr E_X$,
soit $U$ son domaine.
Soit $g$ une carte de~$\mathscr E_Y$ et soit~$V$ son domaine.
D'après l'assertion~\emph b) de la proposition~\ref{prop.lisse-parametre},
l'application linéaire $\gamma\mapsto S_*(\rho_\lambda\gamma)$ 
induit une application
linéaire continue de $\mathscr A^{p,q;0}_{(f,g)}(X\times Y)$
dans~$\mathscr A^0_g(Y)$.
L'assertion en découle.

\item
Soit $\mathscr D$ un jeu de cartes tropicales couvrant sur~$X\times Y$ 
et soit $M$ une partie fermée de~$X\times Y$ qui est compacte au-dessus
de~$Y$. Démontrons que l'application de $\mathscr A^{p,q;0}_{\mathscr D,M}(X\times Y)$ dans~$\mathscr A^0(Y)$ déduite de~$S_*$ est continue.

Soit $(W_i)$ une famille de domaines compacts de~$Y$ dont les intérieurs
recouvrent~$Y$. Soit $\mathscr E_i^X$ et $\mathscr E_i^Y$ des jeux
de cartes tropicales sur~$X$ et~$W_i$ respectivement
tel que le jeu produit $\mathscr E_i^X\times \mathscr E_i^Y$ raffine~$\mathscr D$ et couvre la partie compacte $M\cap (X\times W_i)$.
On prouve comme précédemment que pour toute forme lisse $\gamma\in\mathscr A^{p,q;0}_{\mathscr D,M}(X\times Y)$ 
et toute fonction lisse~$\rho$ sur~$X$,
la fonction $S_*(\rho\gamma)$ est lisse appartient
à $\mathscr A^0_{\mathscr E_i^Y}$ pour tout $i$.
Soit $\mathscr E$ le jeu de cartes tropicales sur~$Y$ 
réunion des jeux~$\mathscr E_i^Y$. Comme les intérieurs des~$W_i$
recouvrent~$Y$, ce jeu est couvrant, et  $S_*$ applique
$\mathscr A^{p,q;0}_{\mathscr D,M}(X\times Y)$ 
dans~$\mathscr A^0_{\mathscr E}(Y)$.
Il suffit de démontrer 
que l'application linéaire qui s'en déduit par restrictions est continue.

Considérons une carte tropicale~$g$ appartenant à~$\mathscr E$,
soit~$V$ son domaine,
soit $W$ un domaine compact de~$V$ et soit $r$ un entier naturel;
on note $\norm\cdot_{W,r}$ une semi-norme $\mathscr C^r$ correspondante
sur l'espace $\mathscr A^0_{\mathscr E}(Y)$.

L'assertion~\emph a) appliquée à $X\times W$ entraîne la continuité
de l'application linéaire $\gamma\mapsto S_*(\gamma)|_W$,
car $W$ est compact.
Ainsi, pour que $\norm{S_*(\gamma)}_{W,r}$ soit petite, il suffit
que $\gamma$ soit suffisamment proche de~$0$ 
dans~$\mathscr A^0_{\mathscr D,M}(X\times W)$,
et a fortiori, dans $\mathscr A^{p,q;0}_{\mathscr D,M}(X\times Y)$.
Cela conclut la démonstration.\qedhere
\end{enumerate}
\end{proof}

\subsection{}
Soit $\Phi$ la famille de supports sur $X\times Y$ constituée
des parties fermées qui sont compactes au-dessus de~$Y$.
On définit un faisceau sur~$Y$
en associant à tout ouvert~$U$ de~$X$,
le sous-espace de $\mathscr A^{p,q;p',q'}(X\times U)$ constitué
des formes de type $(p,q;p',q')$ sur $X\times U$ 
dont le support est compact au-dessus de~$U$.
On le note $p_{2,*}\mathscr A^{p,q;p',q'}_{X\times Y, \Phi}$.

\begin{theo}\label{theo.integration-fibre}
Soit $S$ un courant de bidimension $(p,q)$ sur $X$.
Soit $\Phi$ la famille de supports sur $X\times Y$ constituée
des parties fermées qui sont compactes au-dessus de~$Y$.

\begin{enumerate}
\item
Il existe un unique morphisme de faisceaux sur~$Y$ : 
\[ p_{2,*}\mathscr A^{p,q;p',q'}_{X\times Y,\Phi}\to \mathscr A^{p',q'}_Y, \]
qui, pour tout ouvert~$U$ de~$Y$
et toute toute forme $\gamma$ de type $(p,q;0)$ sur~$X\times U$
dont le support est compact au-dessus de~$U$
et toute forme $\beta$ de type $(p',q')$ sur~$U$, 
applique $\gamma\wedge p_2^*\beta$ sur $S_*(\gamma)\beta$.

\item
Ce morphisme de faisceaux est fonctoriel en~$Y$. En particulier,
si $v\colon Y\to Y'$ est un morphisme d'espaces tropicaux
et $\omega$ est une forme sur $X\times Y'$, on a 
$S_*((\id_X\times v)^*\omega)=v^* S_*(\omega)$.

\item
L'application linéaire 
\[ \mathscr A^{p,q;p',q'}_\Phi(X\times Y) \to \mathscr A^{p',q'}(Y) \]
qui s'en déduit par passage aux sections globales est continue.
\item
L'application linéaire 
\[ \mathscr A^{p,q;p',q'}_\cpct(X\times Y) \to \mathscr A^{p',q'}_\cpct(Y) \]
qui s'en déduit par passage aux sections globales est continue.
\end{enumerate}
\end{theo}
Cette dernière application linéaire
coïncide avec l'application~$S_*$ lorsque $(p',q')=(0,0)$;
nous la noterons donc encore~$S_*$.
\begin{proof}
\begin{enumerate}
\item
L'unicité d'un tel morphisme de faisceaux 
résulte de ce que les formes de ce type engendrent 
le faisceau $p_{2,*}\mathscr A^{p,q;p',q'}_{X\times Y,\Phi}$.
En effet, soit $U$ un ouvert de~$Y$ 
et soit $\omega\in\mathscr A^{p,q;p',q'}_\Phi(X\times Y)$.
Soit $y$ un point de~$U$, soit $W$ un domaine compact de~$Y$ contenu dans~$U$
et qui est voisinage de~$y$. La forme $\omega|_{X\times W}$ est alors
à support compact; d'après le corollaire~\ref{coro.formes-produit},
elle est combinaison linéaire de formes du type $\gamma\wedge p_2^*\beta$;
il en résulte que $\omega|_{\mathring W}$ est combinaison linéaire
de formes du type requis.

Il suffit alors, pour démontrer l'existence 
de ce morphisme de faisceaux, de vérifier que si $U$ est un ouvert de~$Y$,
$(\gamma_i)$ une famille finie de formes de type $(p,q;0)$ sur $X\times U$
à supports compacts au-dessus de~$U$, $(\beta_i)$ une famille finie
de formes de type $(p',q')$ sur~$U$ 
telles que  la forme~$\omega=\sum \gamma_i \wedge p_2^*\beta_i$ est nulle,
on a l'égalité 
\[ \sum S_*(\gamma_i) \beta_i=0. \]
Posons $\varpi=\sum S_*(\gamma_i)\beta_i$.

Comme un espace tropical est supposé riche,
il suffit de vérifier que la restriction 
de~$\varpi$ à toute partie squelettique~$\Tau$ de~$U$
est nulle. 
En considérant un G-recouvrement de~$\Tau$,
on peut supposer que $\Tau$ est une cellule~$D$,
dont on note~$d$ la dimension.

Soit $(\eta_\lambda)$ une base standard des $(p',q')$-formes sur~$D$.
Pour tout $i$, on écrit $\beta_i|_D=\sum b_{i,\lambda}\eta_\lambda$,
où les $b_{i,\lambda}$ sont des fonctions lisses sur~$D$.
Pour tout $\lambda$, posons 
\[ \theta_\lambda = \sum_i \gamma_i|_{X\times D} p_2^*b_{i,\lambda} . \]
Les~$\theta_\lambda$ sont des formes lisses sur~$X \times D$
telles que $\omega|_{X\times D} 
  = \sum_\lambda \theta_\lambda \wedge p_2^*\eta_\lambda$.
Avec ces notations, on a donc
\[ \varpi|_D = \sum_{i,\lambda} S_*(\gamma_i)\beta_i|_D
=\sum_{i,\lambda} S_*(\gamma_i) b_{i,\lambda} \eta_\lambda
=\sum_{i,\lambda} S_*(\gamma_i  p_2^*b_{i,\lambda}) \eta_\lambda
= \sum_\lambda S_*(\theta_\lambda)\eta_\lambda. \]
Puisque $\omega|_{X\times D}=0$, il en résulte que $\theta_\lambda=0$
pour tout $\lambda$. Cela prouve donc que $S_*(\theta_\lambda)=0$,
d'où l'égalité $\varpi|_D=0$.
Cela conclut la démonstration de l'existence du morphisme de faisceaux
de l'énoncé.

\item
Cette propriété se vérifie localement. Il suffit alors 
d'établir l'assertion pour une forme $\omega$ du type $\gamma\wedge p_2^*\beta$,
auquel cas cela découle de la définition.

\item
Notons ici $\widetilde S_*$ cette application linéaire.
Soit $\mathscr D$ un jeu de cartes couvrant sur~$X\times Y$
et soit $M$  une partie fermée de~$X\times Y$ qui est compacte
au-dessus de~$Y$.  Soit $N$ la projection de~$M$ sur~$Y$;
c'est une partie fermée de~$Y$.
Nous allons construire un jeu~$\mathscr E'$ sur~$Y$
tel que l'image de~$\mathscr A^{p,q;p',q'}_{\mathscr D,M}(X\times Y)$
soit contenue dans $\mathscr A^{p',q'}_{\mathscr E,N}(Y)$.

Soit $(W_i)$ une famille de domaines compacts de~$Y$ 
dont les intérieurs recouvrent~$Y$.

Fixons un indice~$i$.
Compte tenu de la proposition~\ref{prop.decoupage-produit},
considérons 
des jeux de cartes tropicales~$\mathscr E_i^X$ sur~$X$ et~$\mathscr E_i^Y$ sur~$W_i$ 
tels que le jeu produit $\mathscr E_i^X\times\mathscr E_i^Y$
raffine~$\mathscr D$,
où le jeu $\mathscr E_i^X$ couvre  la partie compacte 
$K_i=p_1(M\cap (X\times W_i))$ et le jeu $\mathscr E_i^Y$ couvre~$W_i$.
 
Soit $\mathscr E$ le jeu de cartes tropicales sur~$Y$
obtenu par concaténation des jeux~$\mathscr E_i^Y$;
il est couvrant.
Prouvons  que $\widetilde S_*$ applique
$\mathscr A^{p,q;p',q'}_{\mathscr D,M}(X\times Y)$
dans $\mathscr A^{p',q'}_{\mathscr E, N}(Y)$.
Soit donc $\omega\in \mathscr A^{p,q;p',q'}_{\mathscr D,M}(X\times Y)$.
La forme $\widetilde S_*(\omega)$ est supportée par~$N$;
démontrons qu'elle est tropicalisée par~$\mathscr E$.

Soit $i$ un indice,  soit $g$  une carte de~$\mathscr E_i^Y$ et
soit~$V$ son domaine;
démontrons que $\widetilde S_*(\omega)$ est tropicalisée par~$g$.
Soit $W$ un domaine compact de~$V$. 
Reprenons une décomposition de~$\omega|_{X\times W}$ comme dans le paragraphe~\ref{ss.formes-produit-jeu'}:
\[ \omega = \sum_{\lambda, i} j_!(p_1^*\theta_{\lambda,i} (f_\lambda,g)^*\gamma_{\lambda,i}) p_1^*\widetilde \alpha_{\lambda,i} \wedge p_2^* g^*\beta_{\lambda,i}.\]
Ainsi,
\[ \widetilde S_*(\omega)|_W
 = \sum_{\lambda, i}  S_*(j_! p_1^*\theta_{\lambda,i}(f_\lambda,g)^* \gamma_{\lambda,i}) p_1^*\widetilde \alpha_{\lambda,i}) g^*\beta_{\lambda,i} . \]
D'après la proposition~\ref{prop.lisse-parametre}, \emph a), 
les fonctions 
\[ S_*(j_! p_1^*\theta_{\lambda,i}(f_\lambda,g)^* \gamma_{\lambda,i}) p_1^*\widetilde \alpha_{\lambda,i}) \]
sur~$W$ sont lisses, tropicalisées par~$g$.
Il en résulte que $\widetilde S_*(\omega)|_W$ est tropicalisée par~$g$.

Démontrons que l'application linéaire 
\[ \widetilde S_*\colon \mathscr A^{p,q;p',q'}_{\mathscr D,M}(X\times Y)
 \to \mathscr A^{p',q'}_{\mathscr E,N} (Y) \]
est continue.

Soit $D$ une cellule compacte de~$Y$ 
qui est contenue dans le domaine d'une carte~$g$ de~$\mathscr E$
et telle que $g|_D$ soit injective.
Choisissons une base standard $(\eta_\mu)$ des $(p',q')$-formes sur~$D$.
Soit $r$ un entier naturel.
Choisissons la norme~$\mathscr C^r$ standard sur $\mathscr A^{p',q'}_\groth(D)$
donnée par la borne supérieure des normes~$\mathscr C^r$ des coefficients
dans la base~$(\eta_\mu)$.

On choisit un domaine compact~$W$ de~$Y$ contenu dans le domaine de~$g$
tel que $D\subset W$.
Choisissons une décomposition de~$\omega|_{X\times W}$ comme ci-dessus:
\[ \omega|_{X\times W} = \sum_{\lambda, i} j_!(p_1^*\theta_{\lambda,i} (f_\lambda,g)^*\gamma_{\lambda,i}) p_1^*\widetilde \alpha_{\lambda,i} \wedge p_2^* g^*\beta_{\lambda,i}.\]
Ainsi,
\[ \widetilde S_*(\omega)|_W
 = \sum_{\lambda, i}  S_*(j_! p_1^*\theta_{\lambda,i}(f_\lambda,g)^* \gamma_{\lambda,i}) p_1^*\widetilde \alpha_{\lambda,i}) g^*\beta_{\lambda,i} . \]
On peut restreindre ces décompositions sur~$X\times D$.
Pour tout $(\lambda,i)$, il existe une unique famille $(b_{\lambda,i,\mu})$
de fonctions lisses sur~$D$ telle que 
\[ g^*\beta_{\lambda,i}|_D = \sum_{\mu} b_{\lambda,i,\mu} \eta_\mu.  \]
Pour tout $\mu$, posons 
\[ \theta_\mu = 
 \sum_{\lambda, i}  j_! (p_1^*\theta_{\lambda,i} (f_\lambda, \id^*) \gamma_{\lambda,i} p_1^* \widetilde\alpha_{\lambda,i}) p_2^* b_{\lambda,i,\mu} . \]
Ce sont des formes lisses de type $(p,q;0,0)$ sur $X\times D$ telles
que 
\[ \omega|_{X\times D} = \sum_\mu \theta_\mu  \wedge p_2^* \eta_\mu. \]
Ainsi 
\[ \widetilde S_*(\omega)|_D = \widetilde S_*(\omega|_{X\times D}) 
 = \sum_\mu S_*(\theta_\mu)  \eta_\mu. \]

Comme $S_*$ est continu (théorème~\ref{theo.S*-continu}),
il résulte 
de la remarque~\ref{rema.topologie-formes-squelettes}
qu'il existe une partie squelettique compacte de $X\times D$,
qu'on peut prendre de la forme $\Sigma\times D$, 
où $\Sigma$ est une partie squelettique compacte de~$X$,
un entier~$s$, une norme~$\mathscr C^s$ standard sur $\Sigma\times D$
et un nombre réel~$c$
tels $\abs{S_*(\theta_\mu)}_{D,r}\leq c \norm{\theta_\mu}_{\Sigma\times D,s}$.
On a donc
\[ \norm{\widetilde S_*(\omega)}_{D,r} 
   =\sup_\mu \norm{ S_*(\theta_\mu)}_{D,r} 
\leq c \sup_\mu \norm{\theta_\mu}_{\Sigma\times D,s}
= c \norm{\omega}_{\Sigma\times D, s} \]
pour une norme standard $\mathscr C^s$ sur $\Sigma\times D$.
 
Compte tenu de la remarque~\ref{rema.topologie-formes-squelettes},
cela prouve la continuité annoncée.

Ceci entraîne les assertions~\emph c) et~\emph d) 
par définition des topologies  en jeu.
\end{enumerate}
\end{proof}

\section{Produit tensoriel de courants}

\subsection{}
Soit $X$ un espace tropical de niveau~$m$,
soit $Y$ un espace tropical de niveau~$n$,
supposés sans bord.
On a défini au paragraphe~\ref{sect.prod-trop} un espace
tropical produit~$X\times Y$, de niveau~$m+n$,
tel que les projections $p_1\colon X\times Y\to X$ 
et $p_2\colon X\times Y\to Y$ soient des morphismes d'espaces
tropicaux.

\subsection{}
Soit $S$ un courant de bidimension~$(p,q)$ sur~$X$
et soit $T$ un courant de bidimension~$(r,s)$ sur~$Y$.
Le but de ce paragraphe est de définir un courant $S\boxtimes T$
de bidimension~$(p+r,q+s)$ sur l'espace tropical~$X\times Y$
tel que l'on ait $\langle S\boxtimes T,p_1^*\alpha\wedge p_2^*\beta\rangle
= \langle S,\alpha\rangle \langle T,\beta\rangle$
lorsque $\alpha$ et~$\beta$ sont des formes lisses à support compact
sur~$X$ et~$Y$ respectivement, 
de bidegrés~$(p,q)$ et~$(r,s)$ respectivement.
Par construction, ce courant sera nul 
sur les autres facteurs que 
$\mathscr A^{p,q}_{X}\boxtimes \mathscr A^{r,s}_{Y}$.

\subsection{}
Dans ce qui suit, il sera pratique d'évaluer un courant
de type~$(p,q)$ en une forme dont le type est arbitraire;
par convention, sa valeur sera~$0$ sur l'espace des formes
de bidegré~$(p',q')$, pour tout couple~$(p',q')$ distinct de~$(p,q)$.

 
\begin{theo}\label{prop.ST-produit}
Le courant $S\boxtimes T= T \circ S_*$ sur $X\times Y$
vérifie les relations 
\[ \langle S\boxtimes T, p_1^*\alpha \wedge p_2^*\beta\rangle
 = \langle S, \alpha\rangle \langle T, \beta\rangle \]
pour toute forme lisse $\alpha$  à support compact sur~$X$
et toute forme lisse~$\beta$ à support compact sur~$Y$,
et c'est le seul.
\end{theo}
\begin{proof}
Il résulte du corollaire~\ref{coro.densite-formes-produit}
qu'il existe au plus un courant sur~$X\times Y$ qui vérifie ces
relations. Il suffit donc de prouver que c'est le cas de $T\circ S_*$.

Puisque les applications continues $T$ et $S_*$ 
sont continues, l'application $T\circ S_*$ est un courant sur $X\times Y$.
Démontrons qu'il vérifie la relation indiquée.
Soit $\gamma$ une fonction lisse à support compact sur~$X\times Y$
qui vaut $1$ au voisinage du produit des supports de~$\alpha$ et de~$\beta$.
On a donc $p_1^*\alpha\wedge p_2^*\beta= \gamma p_1^*\alpha \wedge p_2^*\beta$,
si bien que
\[ \langle S \boxtimes T, p_1^*\alpha \wedge p_2^*\beta\rangle
 =  \langle S \boxtimes T, \gamma p_1^*\alpha \wedge p_2^*\beta\rangle
= \langle T, \langle S, \gamma(\cdot,y)\alpha\rangle \beta\rangle. \]
Pour démontrer l'égalité voulue, il suffit donc de démontrer que
\[ \langle S, \gamma(\cdot,y)\alpha\rangle \beta 
 = \langle S,\alpha\rangle \beta. \]
Il suffit de le prouver au voisinage de tout point~$y$ de~$Y$.
Si $y\not\in\supp(\beta)$,  la forme~$\beta$ est nulle au voisinage de~$y$, 
donc les deux membres de l'égalité sont nuls au voisinage de~$y$.
Sinon, 
la fonction $\gamma(\cdot,y)$ sur~$X$ vaut~$1$ au voisinage du support de~$\alpha$, si bien que $\gamma(\cdot,y)\alpha=\alpha$. 
On a donc $\langle S,\gamma(\cdot,y)\alpha\rangle=\langle S,\alpha\rangle$
dans ce cas, d'où l'égalité voulue.
\end{proof}

\begin{coro}\label{coro.ST-symetrique}
La construction du courant $S\boxtimes T$ est symétrique: 
si $\tau$ désigne l'échange des facteurs de~$X\times Y$, on a 
$\langle S\boxtimes T,  \omega \rangle 
    = \langle T \boxtimes S, \tau^*\omega\rangle$,
pour toute forme lisse à support compact sur $X\times Y$.
\end{coro}
\begin{proof}
D'après la proposition précédente,
si $\alpha$ et $\beta$ sont des formes lisses à support compact sur~$X$
et~$Y$ respectivement, on a 
\[ \langle S\boxtimes T, p_1^*\alpha\wedge p_2^*\beta\rangle
= \langle S,\alpha \rangle \langle T,\beta\rangle
= \langle T\boxtimes S, p_1^*\beta\wedge p_2^*\alpha\rangle. \]
L'assertion d'unicité de la proposition précédente permet de conclure.
\end{proof}

\begin{exem}
Soit $\omega$ une forme lisse sur~$X$, 
soit $\eta$ une forme lisse sur~$Y$. On a 
\[ [\omega] \boxtimes [\eta] = [p_1^*\omega \wedge p_2^*\eta]. \]

Posons $S=[\omega]$.
Il suffit de démontrer que si $\alpha$ est une forme lisse à support
compact sur~$X$, $\beta$ une forme lisse à support compact sur~$Y$
et $\gamma$ une fonction lisse à support compact sur~$X\times Y$,
on a 
\[ \int_{X\times Y} \gamma p_1^*(\omega\wedge\alpha)\wedge p_2^*(\eta\wedge\beta) = \int_Y S_*(\gamma p_1^*\alpha)  \eta \wedge \beta. \]
Soit $\Sigma$ une partie squelettique compacte de~$X$
qui supporte fortement $\omega\wedge\alpha$ 
et soit $\Tau$ une partie squelettique compacte de~$Y$
qui supporte fortement $\eta\wedge\beta$.
Alors $\Sigma\times \Tau$ supporte fortement
$p_1^*(\omega\wedge\alpha)\wedge p_2^*(\eta\wedge\beta)$, si bien
que 
\[ 
\int_{X\times Y} \gamma p_1^*(\omega\wedge\alpha)\wedge p_2^*(\eta\wedge\beta) =
\int_{\Sigma\times\Tau} \gamma p_1^*(\omega\wedge\alpha)\wedge p_2^*(\eta\wedge\beta) . \]
En vertu du théorème de Fubini, cette intégrale vaut
\[ \int_\Tau \left(\int_\Sigma \gamma (\cdot,y) (\omega\wedge\alpha)\right)
   \eta\wedge\beta ,\]
où $y$ désigne la variable de~$\Tau$.
Par construction, pour tout $y\in Y$, on a égalemement
\[ S_*(\gamma p_1^*\alpha) (y) = \int_{X} \gamma (\cdot,y) (\omega\wedge\alpha)
 = \int_\Sigma \gamma (\cdot, y) (\omega\wedge \alpha), \]
car $\Sigma$ supporte fortement $\omega\wedge\alpha$.
D'autre part, $\Tau$ supporte fortement le produit
$S_*(\gamma p_1^*\alpha) \eta\wedge \beta$, de sorte que
\[ \int_Y S_*(\gamma p_1 p^*\alpha)  \eta\wedge\beta
 = \int_\Tau S_*(\gamma p_1 p^*\alpha) \eta\wedge\beta. \]
Le résultat en découle.
\end{exem}

\begin{prop}\label{prop.SboxT-squelettique}
Si $S$ et $T$ sont des courants squelettiques, alors
$S\boxtimes T$ est un courant squelettique.
\end{prop}
\begin{proof}
Soit $W$ un ouvert de~$X$ et soit $h$ une carte tropicale sur~$X$.
Il s'agit de démontrer que $S\boxtimes T$ est localement descriptible
sur~$W$. Par définition de la structure d'espace tropical sur~$X\times Y$,
on peut ainsi supposer qu'il existe des parties ouvertes $U$ de~$X$
et~$V$ de~$Y$ telles que $W=U\times V$, une carte tropicale~$f$ sur~$U$
et une carte tropicale~$g$ sur~$V$, telles que $h=p\circ (f,g)$ 
où $p$ est une application affine.
On peut également supposer que $S|_U$ est descriptible par une
partie squelettique~$P_f$ munie d'un calibrage G-localement
constant, et que $T|_V$ est descriptible par une partie squelettique~$Q_g$
munie d'un calibrage G-localement constant.

Vérifions que $S\boxtimes T|_{U\times V}$ est descriptible
relativement au moment~$(f,g)$  (et donc relativement au moment~$h$)
par la partie squelettique $P_f\times Q_g$ munie du calibrage produit.
Soit $\omega$ une forme lisse sur~$U\times V$, tropicalisée par~$(f,g)$
et soit $\rho$ une fonction lisse à support compact sur $U\times V$. 
Il s'agit de démontrer que 
\[ \langle S \boxtimes T, \rho \omega\rangle
 = \int_{P_f\times Q_f} \rho \omega|_{P_f\times Q_g}. \]
En appliquant le paragraphe~\ref{ss.formes-produit-jeu}
au jeux réduits aux cartes $(U,f)$ et $(V,g)$, on voit que
$\omega$ est somme finie de formes du type
$\gamma p_1^*\alpha\wedge p_2^*\beta$,
où $\alpha$ est tropicalisée par~$f$ sur un ouvert~$U'$ de~$U$,
$\beta$ est tropicalisée par~$g$ sur un ouvert~$V'$ de~$V$
et $\gamma$ est à support compact dans $U'\times V'$.
On peut donc supposer que $\rho\omega$ est de ce type;
compte tenu du lemme~\ref{lemm.courant-paralineaire-restriction},
on peut supposer que $U'=U$ et $V'=V$.
La fonction $S_*(\gamma p_1^*\alpha)$ est à support compact dans~$U$
(lemme~\ref{lemm.support-S*})
et la forme $\beta$ est tropicalisée par~$g$.
On a donc
\[ \langle S \boxtimes T, \rho\omega\rangle
= \langle T, S_*(\gamma p_1^*\alpha)\rangle
 = \int_{Q_g} S_*(\gamma p_1^*\alpha) \beta. \]
Par ailleurs, pour tout $y\in Y$, on a 
\[ S_*(\gamma p_1^*\alpha) (y) = \langle S , \gamma(\cdot,y) \alpha\rangle
 = \int_{P_f} \gamma(\cdot,y)\alpha|_{P_f} \]
puisque $\alpha$ est tropicalisée par~$f$ sur~$U$ et $\gamma(\cdot,y)$
est à support compact sur~$U$.
Ainsi,
\[ \langle S \boxtimes T, \rho\omega\rangle
= \int_{Q_g} \left(\int_{P_f} \gamma(\cdot,y)\alpha|_{P_f} \right) \beta|_{Q_g}
= \int_{P_f\times Q_g} \gamma p_1^*\alpha \wedge p_2^*\beta|_{P_f\times Q_g}
\]
en vertu du théorème de Fubini.
Cela conclut la démonstration.
\end{proof}

\begin{prop}\label{prop.box-di-dc}
Soit $S$ et $T$ des courants sur~$X$ et~$Y$. 

On a $\di[X](S\boxtimes T)=(\di[X] S)\boxtimes T$
et $\di[Y](S\boxtimes T)=S \boxtimes \di[Y](T)$.
En particulier, si $S$ et $T$ sont $\di$-fermés,
$S\boxtimes T$ est $\di$-fermé.

Les assertions analogues valent pour les opérateurs~$\dc$.
\end{prop}
\begin{proof}
Il suffit de vérifier les relations demandées contre une forme-test
du type $p_1^*\alpha\wedge p_2^*\beta$, où $\alpha$
et $\beta$ sont des formes lisses à support compact sur~$X$ et~$Y$.
En utilisant la proposition~\ref{prop.ST-produit} à plusieurs reprises, il vient
\begin{align*}
 \langle \di[X](S\boxtimes T),p_1^*\alpha\wedge p_2^*\beta \rangle
&= \pm \langle S\boxtimes T, \di[X](p_1^*\alpha\wedge p_2^*\beta)\rangle \\
&= \pm \langle S\boxtimes T, p_1^*(\di\alpha)\wedge p_2^*\beta\rangle \\
&= \pm \langle S, \di\alpha\rangle \langle T,\beta \rangle  \\
&= \langle \di S,\alpha \rangle \langle T, \beta \rangle \\
&= \langle (\di[X] S)\boxtimes T, p_1^*\alpha \wedge p_2^*\beta\rangle. \end{align*}
Les autres formules se prouvent de même.

Si $S$ et $T$ sont $\di$-fermés,
alors $\di(S\boxtimes T)=\di[X](S\boxtimes T)+\di[Y](S\boxtimes T)
= (\di S)\boxtimes T+S\boxtimes(\di T)=0$.
\end{proof}

\begin{prop}\label{prop.box-positif}
Si $S$ et $T$ sont des courants positifs (\resp faiblement positifs,
\resp fortement positifs), il en est de même du courant
$S\boxtimes T$.
\end{prop}
\begin{proof}
Notons $(p,p)$ et $(q,q)$ les bidimensions de~$S$ et~$T$.
Soit $\omega$ une forme lisse positive à support compact sur $X\times Y$,
de bidegré $(p,p;q,q)$;
il s'agit de prouver que $\langle S\boxtimes T,\omega\rangle$
est positif. 
Il suffit pour cela de démontrer que la forme $\widetilde S_*(\omega)$
sur~$Y$ est positive.
Puisque $Y$ est très riche, il s'agit alors de démontrer
que le courant $[\widetilde S_*(\omega)]$ est positif.
Soit donc $\eta$ une forme positive à support compact, de bidegré~$(n-q,n-q)$;
prouvons que $\int_Y \widetilde S_*(\omega)\wedge \eta\geq 0$.
Il suffit pour cela de démontrer que la forme $\widetilde S_*(\omega) \wedge\eta$
est positive. Or, par construction, on a
$\widetilde S_*(\omega)\wedge\eta=\widetilde S_*(\omega\wedge p_2^*\eta)$.
Autrement dit, on peut supposer que $\omega$ est de type $(p,q;n,n)$;
posons $\varpi=\widetilde S_*(\omega)$. 
Soit $C$ une partie squelettique compacte de~$Y$ qui est une $n$-cellule;
démontrons que $\varpi|_C$ est positive;
La construction de $\varpi$ écrit $\omega|_{X\times C}=\theta p_2^*\nu$,
où $\nu$ est une $(n,n)$-forme standard sur~$C$
et $\theta$ une $(p,p;0,0)$-forme lisse à support compact sur $X\times C$,
et l'on a $\varpi|_C=S_*(\theta)\nu$.
La positivité de~$\omega$ entraîne que pour tout $y\in C$,
la restriction de $\theta$ à~$X\times\{y\}$ est une $(p,p)$-forme positive;
comme $S$ est un courant positif, on a 
\[ S_*(\theta)(y)=\langle S, \theta(\cdot;y)\rangle \geq 0. \]
Ainsi, $\varpi|_C$ est positive, et cela conclut la démonstration
que $S\boxtimes T$ est un courant positif si $S$ et $T$ le sont.

Les deux autres assertions se prouvent de même.
\end{proof}

\begin{prop}\label{prop.hemicontinuite}
\begin{enumerate}
\item
Soit $\mathscr S$ un filtre de courants sur~$X$ 
qui converge vers un courant~$S$ et soit $T$ un courant sur~$Y$.
Alors le filtre de courants sur~$X\times Y$
engendré par $\mathscr S\boxtimes T$ 
converge vers $S\boxtimes T$.
\item
Soit $S$ un courant sur~$X$ 
et soit $\mathscr T$ un filtre de courants sur~$Y$
qui converge vers un courant~$T$.
Alors le filtre de courants sur~$X\times Y$
engendré par $S\boxtimes \mathscr T$ 
converge vers $S\boxtimes T$.
\end{enumerate}
\end{prop}
\begin{proof}
Les deux assertions étant symétriques, le corollaire~\ref{coro.ST-symetrique}
nous permet de n'établir que la seconde.
Il s'agit de démontrer
que le filtre sur~$\R$
engendré par $\langle S\boxtimes \mathscr T,\omega\rangle$ 
converge vers $\langle S\boxtimes T,\omega\rangle$,
pour toute forme lisse à support compact~$\omega$ sur~$X\times Y$.
D'après le corollaire~\ref{coro.formes-produit},
on peut se limiter au cas où $\omega=
\gamma p_1^*\alpha \wedge p_2^*\beta$,
où $\gamma$ est une fonction lisse à support compact sur~$X\times Y$,
$\alpha$ est une forme lisse à support compact sur~$X$ et 
$\beta$ est une forme lisse à support compact sur~$Y$.
Pour tout courant~$U$ sur~$Y$, on a 
\[ \langle S\boxtimes U,\omega \rangle
 = \langle U, S_*(\gamma p_1^*\alpha) \beta \rangle. \]
La forme sur~$Y$ à laquelle le courant~$U$ est appliqué est constante,
d'où l'assertion.
\end{proof}

\section{Morphisme de Gysin}

\subsection{}
Soit $X$ un espace tropical de niveau~$m$,
soit $f\colon X\to \R^n$ une application paralinéaire.
Soit $p\colon X\times\R^n \to X$ la première projection et
soit $j\colon X\to X\times \R^n$ l'application $x\mapsto (x,f(x))$;
l'image de~$j$ est le graphe~$\Gamma_f$ de~$f$.

Pour tout $i\in\{1,\dots,n\}$, on définit une fonction 
paralinéaire~$u_i$
sur~$X\times\R^n$ par $u_i(x,y)=\max(f_i(x),y_i)$.
Si $f_i$ est psh, il en est de même de~$u_i$ (proposition~\ref{exem.psh-approchable});
c'est en particulier le cas si $f_i$ est affine.
Dans le cas général, $f_i$ est localement
différence de deux fonctions tropicalement convexes,
et est en particulier dpsh (proposition~\ref{rema.G-lisse-diff-psh}).

Soit $S$ un courant positif fermé sur~$X\times\R^n$. On pose alors
\[ S\cap \Gamma_f  = \ddc (u_1 \ddc(u_2\dots \ddc(u_n S))); \]
c'est un courant sur~$X\times\R^n$,
qu'on a noté plus simplement
\[ S \wedge \bigwedge_{i=1}^n \ddc (u_i) \]
lorsque $S$ est un courant de Bedford-Taylor.

\begin{prop}
On suppose que les~$f_i$ sont harmoniques.
Alors, le courant $S\cap\Gamma_f$ est positif fermé.
\end{prop}
\begin{proof}
Par construction, on a $u_i=g_i^*v$, 
où $g_i\colon X\times\R^n$
est la carte harmonique $(x,y)\mapsto (f_i(x), y_i)$
et $v$ la fonction convexe $(x,y)\mapsto \max(x,y)$.
Par récurrence sur~$n$, le corollaire~\ref{coro.wedge-weak-psh}
entraîne que $S\cap\Gamma_f$ est positif.
Il est fermé par construction.
\end{proof}

\begin{rema}
Par construction de l'intersection de Bedford--Taylor,
le courant $S\cap\Gamma_f$ est positif fermé
dès que les fonctions~$u_i$ sont psh.
Comme elles sont paralinéaires, elles sont dpsh,
si bien que ce courant est localement différence 
de deux courants positifs fermés.

Par ailleurs, nous démontrerons dans le cas des espaces analytiques
qu'une application $\Q$-\pl harmonique est psh
(corollaire~\ref{coro.psh-harmo}),
mais il est peu probable que cette coïncidence soit vraie
dans le cadre général des espaces tropicaux.
Il serait cependant possible d'étendre la notion
de fonction psh de sorte à y adjoindre les fonctions harmoniques.
\end{rema}

Si $S$ est un courant squelettique, le courant $S\cap\Gamma_f$
est squelettique (théorème~\ref{theo.ddcuT-paralineaire}).

Si $S$ est un courant de Bedford-Taylor, le courant $S\cap\Gamma_f$
est de Bedford-Taylor (\S\ref{ss.bt-def}). 

Cette construction est continue en~$S$ (pour la convergence faible
des courants) et en~$f$ (pour la convergence localement uniforme des
applications paralinéaires).

\begin{lemm}\label{lemm.support-S-inter-Gamma}
Soit $S$ un courant positif fermé sur $X\times\R^n$.
Si $f$ est harmonique, alors le support du courant 
$ S \cap \Gamma_f$ est contenu dans le graphe de~$f$.
\end{lemm}
\begin{proof}
Pour tout $k\in\{1,\dots,n\}$, soit $V_k$
l'ensemble des $(x,y)\in X\times \R^n$ tels que $f_k(x)\neq y_k$;
c'est une partie ouverte de~$X\times\R^n$.
La réunion des~$V_k$ est le complémentaire du graphe de~$f$.
Il suffit donc de démontrer que pour tout~$k$, la restriction à~$V_k$
de $S\cap \Gamma_f$ est nulle.

Soit $k\in\{1,\dots,n\}$ et soit $(x,y)\in V_k$.
Si $f_k(x)<y_k$, alors $u_k$ coïncide avec la $k$-ième fonction coordonnée
du second facteur au voisinage de~$(x,y)$, qui est une application affine ;
sinon, $u_k$ coïncide avec la fonction~$f_k$ au voisinage de~$(x,y)$.

Si $f_k$ est affine, cela prouve que la fonction~$u_k$ est affine,
donc lisse, sur $V_k$, et vérifie $\ddc u_k=0$.
En restriction à~$V_k$, on a donc
\[ \ddc (u_k \ddc(u_{k+1} \dots \ddc (u_n S)))= \ddc u_k \wedge
 \ddc(u_{k+1} \dots \ddc (u_n S))=0. \]
A fortiori, $S\cap \Gamma_f$ est nulle sur~$V_k$.

Dans le cas où $f_k$ est harmonique, cela montre que $u_k|_{V_k}$
est paralinéaire, harmonique, et
que la restriction à~$V_k$ du courant $\ddc [u_k]$ est nulle.
Comme le courant $ \ddc(u_{k+1} \dots \ddc (u_n S))$
est de Bedford--Taylor (car $S$ l'est, \S\ref{ss.bt-def}), 
on a donc $\ddc(u_k \ddc(u_{k+1} \dots \ddc (u_n S)))=0$ sur~$V_k$, 
puis $S\cap[\Gamma_f]=0$ sur~$V_k$.
\end{proof}

\begin{defi}
Soit $X$ un espace tropical de niveau~$m$,
soit $f\colon X\to \R^n$ une application harmonique,
soit $p\colon X\times\R^n\to X$ la première projection
et soit~$j\colon X\to X\times \R^n$ l'application 
$x\mapsto (x,f(x))$.
Si $S$ est un courant positif fermé~$S$ sur~$X\times \R^n$,
on pose 
\[ j^!S = p_* (S \cap \Gamma_f). \]
\end{defi}
Cette définition est légitime. En effet,
d'après la proposition précédente,
le support du courant $S\cap\Gamma_f$ est contenu
dans le graphe~$\Gamma_f$ de~$f$; il est donc
compact au-dessus de~$X$.

\begin{prop}\label{prop.j!-squelettique}
Soit $f\colon X\to\R^n$ une carte harmonique
et soit $S$ un courant positif fermé sur~$X\times\R^n$.
Si le courant~$S$ est squelettique, 
il en est de même du courant~$j^!S$.
\end{prop}
\begin{proof}
Il découle d'abord du corollaire~\ref{coro.ddcuT-paralineaire}
que le courant $S\cap\Gamma_f$ est squelettique.
Son support est contenue dans la partie squelettique~$\Gamma_f$
de~$X\times\R^n$ (lemme~\ref{lemm.support-S-inter-Gamma}).

Soit $U$ un ouvert de~$X$, soit $g$ une carte tropicale sur~$U$.
L'application $g\circ p$ est une carte tropicale sur $p^{-1}(U)$.
Soit $P$ la partie squelettique calibrée $P_{g\circ p}^{S\cap\Gamma_f}$
de~$X\times\R^n$, minimale, qui décrit $S\cap\Gamma_f$
relativement à cette carte tropicale. 
Elle est contenue dans le support de~$S\cap\Gamma_f$,
donc dans~$\Gamma_f$.
Soit $Q=p(P)$ son image dans~$X$ par l'isomorphisme $p|_{\Gamma_f}$;
c'est une partie squelettique calibrée de~$X$.
Vérifions qu'elle décrit le courant $j^!S$
relativement à la carte~$g$.
Soit $\omega$ une forme lisse $g$-tropicale sur~$U$
et soit $\rho$ une fonction lisse sur~$U$
telles que $\rho\omega$ soit à support compact.
Soit $\theta$ une fonction lisse sur~$p^{-1}(U)$
telle que $\theta\cdot p^*(\rho\omega)$ soit à support compact
et qui vaut~$1$ au voisinage de~$\supp(p^*\rho\omega)\cap \Gamma_f$.
On a alors
\begin{align*} \langle j^!S, \rho\omega\rangle
& = \langle p_* (S\cap \Gamma_f),\rho\omega\rangle
= \langle S\cap\Gamma_f, \theta p^*(\rho\omega)\rangle
= \langle S\cap\Gamma_f, (\theta p^*\rho) p^*\omega \rangle \\
& = \int_{P} (\theta p^*\rho) p^*\omega 
= \int_Q \rho\omega, \end{align*}
d'où l'assertion.
\end{proof}

\begin{prop}
Soit $f\colon X\to\R^n$ une application paralinéaire harmonique.
L'application~$j$ est un morphisme compact d'espaces G-tropicaux.
Pour tout courant positif fermé~$S$  sur~$X\times\R^n$ et 
toute G-forme lisse~$\omega$ sur~$X\times\R^n$,
à support compact,
on a $\langle S\cap \Gamma_f, \omega\rangle = \langle j^!S, j^*\omega\rangle$.
\end{prop}

\begin{coro}
Si $f$ est affine, 
l'application~$j$ est un morphisme compact d'espaces tropicaux et l'on a
l'égalité $j_*(j^! S )=S\cap\Gamma_f$.
\end{coro}

\begin{proof}[Démonstration de la proposition]
Comme $f$ est paralinéaire, l'application~$j$ est un morphisme d'espaces G-tropicaux.
Il induit un homéomorphisme de~$X$ sur le graphe de~$f$; il est en particulier compact.
Notons $p$ et $q$ les projections de~$X\times\R^n$ sur $X$ et~$\R^n$.

Soit $d$ l'entier tel que $S$ est de bidimension~$(d,d)$
et soit $\omega$ une G-forme G-lisse à support compact sur~$X\times\R^n$
de bidegré~$(d,d)$.
On a 
\[ \langle j^!S, j^*\omega\rangle = \langle S\cap \Gamma_f, p^*j^*\omega \rangle, \]
où cette dernière expression a un sens car l'intersection
de~$\Gamma_f$ et du support de $p^*j^*\omega$ est compacte.
Il s'agit de vérifier qu'elle est égale 
à $\langle S\cap \Gamma_f,\omega\rangle$.

L'application $(x,y)\mapsto (x,y-f(x))$ est un automorphisme 
paralinéaire de~$X\times\R^n$
qui transforme~$f$ en l'application affine nulle
et $S\cap\Gamma_f$ en $S\cap\Gamma_0$.
On peut ainsi supposer que $f\equiv 0$ et $\Gamma_f=X\times\{0\}$.

Il suffit alors de démontrer l'égalité lorsque $\omega$ est une \emph{forme} lisse;
en effet, cette égalité pour toute forme lisse entraînera l'égalité
de courants $j_*(j^!S)=S\cap\Gamma_f$ sur~$X\times\R^n$; comme il s'agit ce courants 
positifs, on en déduira l'égalité voulue pour toute G-forme.

D'après le corollaire~\ref{coro.formes-produit}, $\omega$ est somme finie
de formes du type $\gamma p^*\alpha \wedge q^*\beta$,
où $\alpha$ est une forme lisse à support compact sur~$X$,
$\beta$ est une forme lisse à support compact sur~$Y$
et $\gamma$ est une fonction lisse à support compact sur~$X\times Y$.
Par additivité, il suffit de prouver l'égalité voulue lorsque
la forme~$\omega$ est de ce type.
On peut aussi supposer que les formes~$\beta$ sont homogènes.

Traitons d'abord le cas où~$\beta$ n'est pas de bidegré~$(0,0)$.
Dans cas, $j^*\beta=0$, car $f\equiv0$, de sorte que $p^*j^*\beta=0$.
Démontrons donc que 
$\langle S\cap\Gamma_f,\gamma p^*\alpha\wedge p^*\beta\rangle=0$.
Par définition,
$ \langle S\cap \Gamma_f, \omega\rangle$
est une limite de nombres réels de la forme 
$ \langle S, p^*(\ddc u_1\wedge\dots\wedge\ddc u_n)\wedge \omega\rangle$
où $u_1,\dots,u_n$ sont des fonctions lisses convexes de~$y_1,\dots,y_n$
respectivement qui convergent vers $\max(y_1,0),\dots, \max(y_n,0)$.
Puisque $\beta$ n'est pas de bidegré~$(0,0)$, 
on a 
\[ \ddc u_1\wedge\dots\wedge\ddc u_n \wedge \beta=0 \]
donc $\langle S\cap\Gamma_f,\omega\rangle=0$.

Traitons maintenant le cas où~$\beta$ est bidegré~$(0,0)$;
on peut supposer $\beta=1$.
L'application qui à une fonction~$u$ sur~$X\times\R^n$,
lisse et à support compact, 
associe $\langle S\cap\Gamma_f, up^*\alpha\rangle$
se prolonge en une mesure signée sur~$X\times\R^n$,
à support dans $\Gamma_f$, 
car le courant $S\cap\Gamma_f$ est positif 
et son support est contenu dans~$\Gamma_f$.
Sa valeur ne dépend donc que de la restriction de~$u$ à~$\Gamma_f$.
On remarque alors que 
\[ p^*j^* (\gamma p^*\alpha) = p^*j^*\gamma p^*j^*p^*\alpha 
=p^*j^*\gamma p^*\alpha, \]
car $p\circ j=\id_X$, 
et que les fonctions $\gamma$ et $p^*j^*\gamma$
coïncident sur $\Gamma_f$. Cela entraîne l'égalité
$\langle S\cap\Gamma_f, \gamma p^*\alpha\rangle
=\langle S\cap\Gamma_f, p*j^*(\gamma p^*\alpha)\rangle $
qu'il fallait démontrer. 
\end{proof}
 
\section{Intersection tropicale}

Soit $X$ un espace tropical de niveau~$m$ et soit $f\colon X\to\R^n$
une application $\R$-paralinéaire harmonique.
Soit $j$ l'immersion du graphe~$\Gamma_f$ de~$f$ dans~$X\times\R^n$.

\begin{defi}
Pour tout courant positif fermé~$S$ sur~$X$ et tout
courant positif fermé~$T$ sur~$\R^n$, on pose
\[ S \cap_f T = j^!(S\boxtimes T). \]
\end{defi}

La définition fait sens car
si $S$ est un courant positif fermé sur~$X$
et $T$ est un courant positif fermé sur~$\R^n$,
le courant $S\boxtimes T$ est encore positif et fermé
(propositions~\ref{prop.box-positif} et~\ref{prop.box-di-dc}),
de sorte qu'on peut appliquer la construction~$j^!$.
Revenant aux définitions,
on a donc
\[ S\cap_f T =  p_*((S\boxtimes T)\cap\Gamma_f).\]

Si $S$ est de bidegré~$(p,q)$ et $T$ est de bidegré~$(q,q)$,
alors le courant $S\boxtimes T$ sur $X\times\R^n$
est de bidegré $(p,p;q,q)$, donc de bidegré $(p+q,p+q)$,
et le courant $(S\boxtimes T)\cap\Gamma_f$ est de bidegré~$(p+q+n,p+q+n)$.
Par suite, le courant $S\cap_f T$ sur~$X$ est de bidegré $(p+q,p+q)$.

\begin{prop}
Si les courants~$S$ et~$T$ sont squelettiques,
le courant $S\cap_f T$ est squelettique.
\end{prop}
\begin{proof}
Le courant $S\boxtimes T$ est squelettique 
d'après la proposition~\ref{prop.SboxT-squelettique}.
Il découle alors de la proposition~\ref{prop.j!-squelettique}
que $S\cap_f T=j^!(S\boxtimes T)$ est squelettique.
\end{proof}

\begin{rema}
Lorsque $X=\R^n$ et $f$ est l'application identique,
la construction $S\cap_f T$ fournit une définition
analytique de l'intersection tropicale. 
Elle est bilinéaire.
La proposition~\ref{prop.intertrop-assoc}
ci-dessous démontre qu'elle est associative,
et la proposition~\ref{prop.intertrop-comm}
prouve qu'elle est commutative.

Cela a été démontré indépendemment par~\cite{Mihatsch-2023}.
La construction donnée dans le texte copie l'approche
de~\cite{AllermannRau-2010}, elle-même inspirée de l'approche
de~\cite{fulton98}.
\end{rema}

\begin{prop}\label{prop.intertrop-forme}
Soit $f\colon X\to\R^n$ une carte harmonique,
soit $S$ un courant positif fermé de bidegré~$(p,p)$ sur~$X$
et soit~$\beta$ une forme lisse positive fermée de bidegré~$(q,q)$ sur~$\R^n$.
On a 
\[ S \cap_f [\beta] = S \wedge f^*\beta.  \]
\end{prop}
\begin{proof}
Notons $p$ et $q$ les projections de~$X\times\R^n$ vers~$X$ et~$\R^n$;
posons aussi $f=(f_1,\dots,f_n)$ et $q=(q_1,\dots,q_n)$.
Soit $\omega$ une forme lisse de bidegré~$(p+q-m,p+q-m)$
à support compact sur~$X$.
Soit $\theta$ une fonction lisse à support compact sur~$\R^n$
telle que $\theta\equiv 1$ au voisinage de $f(\supp(\omega))$.
Alors $q^*\theta\equiv 1$ au voisinage de $(\supp(\omega)\times\R^n)\cap \Gamma_f$
et $p^*\omega\wedge q^*\theta$ est à support compact.
Par définition d'un courant image directe, on a donc
\[ \langle S\cap_f [\beta], \omega\rangle
 = \langle p_*(S\boxtimes[\beta]\cap\Gamma_f,\omega\rangle
= \langle S \boxtimes [\beta] \cap\Gamma_f, p^*\omega \wedge q^*\theta \rangle. \]
Soit $(\mu_\eps)$ une famille de fonctions lisses psh sur~$\R$,
indexée par $\eps\in\R_+^*$,
qui converge uniformément vers $\mu\colon t\mapsto \sup(t,0)$ quand $\eps\to0$.
Pour tout $i\in\{1,\dots,n\}$, posons $u_i=\sup(f_i\circ p,q_i)=q_i+\sup(f_i\circ p-q_i,0)$
et $u_i^\eps=q_i+\mu_\eps\circ (f_i\circ p-q_i)$; les~$u_i$ sont des fonctions continues
sur~$X$, les~$u_i^\eps$ sont des fonctions G-lisses sur~$X\times\R^n$,
harmoniquement tropicalisées,
et l'on a $\lim_{\eps\to0} u_i^\eps=u_i$, uniformément. 
Plus précisément, l'application $\phi_i=f_i\circ p-q_i\colon X\to\R$
est une carte harmonique, et l'on a $u_i^\eps=q_i+\phi_i^*\mu_\eps$,
ainsi que $u_i=q_i+\phi_i^*\mu$. (Les fonctions~$\mu_\eps$ et~$\mu$
sont convexes.)
Par définition, on a 
\[ S\boxtimes[\beta]\cap\Gamma_f 
  = S\boxtimes [\beta] \wedge \bigwedge_{i=1}^n \ddc (q_i+\phi_i^* \mu)
= S\boxtimes[\beta] \wedge \bigwedge _{i=1}^n \ddc \phi_i^*\mu,  \]
car les $q_i$ sont affines.
Compte tenu de la proposition~\ref{prop.BT-continu'},
on a ainsi
\[ S\boxtimes[\beta]\cap\Gamma_f
  = \lim_{\eps\to0} S\boxtimes [\beta] \wedge \bigwedge_{i=1}^n \ddc u_i^\eps. \]
Par suite,
\begin{align*}
  \langle S\cap_f [\beta], \omega\rangle
 =\lim_{\eps\to0} \langle S\boxtimes [\beta] \wedge \bigwedge_{i=1}^n\ddc u_i^\eps, p^*\omega \wedge q^*\theta \rangle .
\end{align*}
Posons
\[ \Omega_\eps=\bigwedge_{i=1}^n\ddc u_i^\eps \wedge p^*\omega \wedge q^*\theta \rangle.\]
C'est une G-forme sur~$X\times\R^n$ qui est tropicalisée par~$f\times\id$
sur tout domaine de la forme $V\times\R^n$, où $V$
est un domaine analytique de~$X$ telle que $f|_V$ est affine
et $\omega|_V$ est tropicale.
On a
\[
\langle S\boxtimes [\beta] \wedge \bigwedge_{i=1}^n\ddc u_i^\eps, p^*\omega \wedge q^*\theta \rangle 
= \langle S\boxtimes [\beta] , \Omega_\eps\rangle.\]
Reprenons l'application linéaire 
\[ p_*\colon \mathscr A^{(p,p;n,n)}_\cpct(X\times\R^n)
\to \mathscr A^{(p,p)}(X) \]
d'intégration dans les fibres de~$p$ définie 
par le théorème~\ref{theo.integration-fibre}.
Pour tout domaine analytique~$V$ de~$X$ sur lequel $f$ est affine
et $\omega$ est tropicale, 
on dispose d'une forme lisse $p_*((\Omega_\eps)|_{V\times\R^n})$ sur~$V$,
qui est tropicalisée par~$f$ puis, par recollement
d'une G-forme G-lisse $p_*(\Omega_\eps)$ sur~$X$.
Compte tenu de la construction du courant~$S\boxtimes[\beta]$,
on a
\[
\langle S\boxtimes [\beta] \wedge \bigwedge_{i=1}^n\ddc u_i^\eps, p^*\omega \wedge q^*\theta_\eps \rangle 
= 
\langle S, p_*(\Omega_\eps \wedge q^*\beta)\rangle.
\]
Observons que 
\[ p_*(\Omega_\eps\wedge q^*\beta) = p_*(\bigwedge_{i=1}^n\ddc u_i^\eps \wedge q^*(\theta\beta)) \wedge \omega. \]
Par continuité de l'extension de~$S$ aux G-formes,
il suffit donc, pour établir la proposition, de démontrer  que les G-formes
\[ 
 p_*\big( \bigwedge_{i=1}^n\ddc u_i^\eps  \wedge q^*(\theta\beta) \big)
\]
sur~$X$ convergent vers $  f^*(\theta\beta) $ lorsque $\eps\to0$.

Pour simplifier les notations, on remplace~$\beta$ par~$\theta\beta$.
Il suffit d'établir cette convergence sur tout domaine sur lequel~$f$
est affine, ce qui ramène au cas où~$X$ est compact et $f$ est affine.

Puisque $u_i^\eps=q_i+\mu_\eps \circ (f_i\circ p-q_i)$, que $q_i$ est affine
et que $f_i\circ p-q_i$ est affine, on a 
\begin{align*} \ddc u_i^\eps & = \ddc ((f_i\circ p-q_i)^* \mu_\eps) \\
& = (f_i\circ p-q_i)^* \mu_\eps''(t)\di t \wedge \dc t \\
& = (f_i\circ p-q_i)^*\mu_\eps'' \cdot (p^*f^*\di y_i-q^*\di y_i)\wedge 
 (p^*f^*\dc y_i - q^*\dc y_i) , \end{align*}
où $(y_1,\dots,y_n)$ sont les fonctions coordonnées de~$\R^n$.
Pour toute partie $I$ de~$\{1,\dots,n\}$,
on note $\di y_I$ le produit des facteurs~$\di y_i$, pour $i\in I$,
et $\di y^I$ le produit des facteurs $\di y_i$, pour $i\not\in I$,
écrits dans l'ordre croissant des indices; on définit de même $\dc y_I$
et $\dc y^I$.
Si $I$ et $J$ sont des parties de $\{1,\dots,n\}$,
notons $\eps_{IJ}$ l'unique élément de~$\{\pm1\}$ tel que
\[ \di y_1\wedge \dc y_1 \wedge \dots
 = \eps_{IJ} \di y_I \wedge \dc y_J \wedge \di y^{I}\wedge \dc y^{J}. \]
En développant le produit $\bigwedge_{i=1}^n \ddc u_i^\eps$,
on obtient
\begin{align*} 
\bigwedge_{i=1}^n \ddc u_i^\eps & = 
\bigwedge_{i=1}^n (f_i\circ p-q_i)^*\mu_\eps''\cdot (p^*f^*\di y_i - q^* \di y_i)
\wedge (p^*f^*\dc y_i - q^*\dc y_i) \\
& = \prod_{i=1}^n (f_i\circ p-q_i)^*\mu_\eps'' \\
& \qquad\qquad
\sum_{I, J \subset\{1,\dots,n\}} 
\eps_{IJ} (-1)^{\Card(I)+\Card(J)}  p^*f^*(\di y^I \wedge \dc y^J)\wedge q^*(\di y_I \wedge \dc y_J) . 
\end{align*}
De là, on déduit que
\begin{align*}
p_*\big( \bigwedge_{i=1}^n\ddc u_i^\eps  \wedge q^*\beta \big)  \hskip -3cm \\
& = \sum_{I, J} \eps_{IJ} (-1)^{\Card(I)+\Card(J)} 
    p_* \big( q^*(\di y_I \wedge \dc y_J \wedge \beta)
        \wedge p^* f^*(\di y^I\wedge \dc y^J)\big) \\
& = \sum_{I, J} \eps_{IJ} (-1)^{\Card(I)+\Card(J)} p_* \big( q^*(\di y_I \wedge \dc y_J \wedge \beta) \big)
    \wedge f^*(\di y^I\wedge \dc y^J).
\end{align*}
La forme $\beta$ est de type~$(q,q)$;
la forme $q^*(\di y_I\wedge \dc y_J\wedge \beta)$ sur $X\times\R^n$
à laquelle est appliquée~$p_*$ 
est donc de type $(0,0;q+\Card(I),q+\Card(J))$.
Seules les parties~$I$ et~$J$ telles que $\Card(I)=\Card(J)=n-q$
interviennent.
Comme $\beta$ est de type~$(q,q)$, elle se décompose 
en 
\[ \beta = \sum_{I,J} \beta_{IJ} \di y^I \wedge \dc y^J , \]
où la somme est précisément sur ces couples de parties~$(I,J)$
telles que $\Card(I)=\Card(J)=n-q$.
En posant $\eta=\di y_1\wedge\dc y_1\wedge\dots$,
l'expression précédente  vaut donc 
\begin{multline*}
\sum_{I,J} p_* \big( \prod_{i=1}^n(q_i-f_i\circ p)^*\mu_\eps'') q^*(\beta_{IJ} \eta) \wedge p ^* f^*(\di y^I\wedge \dc y^J)\big)\\
 =  \sum_{I,J} p_* \big(\prod_i (q_i-f_i\circ p)^*\mu_\eps''\cdot  q^*(\beta_{IJ} \eta )\big) \wedge f^*(\di y^I\wedge \dc y^J)\big).
\end{multline*}
Observons alors que 
$p_* \big( \prod_{i=1}^n (q_i-f_i\circ p)^*\mu_\eps'' q^*(\beta_{IJ} \eta)  \big)$
est la fonction lisse sur~$X$ qui applique $x\in X$ 
sur
\[ \int_{\R^n} \prod_{i=1}^n \mu_\eps''(y_i-f_i(x))
  \beta_{IJ}\eta . \]
Comme $\mu_\eps''\to\delta_0$ lorsque $\eps\to 0$, 
on en déduit que cette expression converge uniformément,
avec toutes ses dérivées,
vers $ \beta_{IJ}(f(x))$,
c'est-à-dire vers la fonction lisse~$f^*(\beta_{IJ})$ sur~$X$.
Par suite,
l'expression qui nous intéresse converge pour la topologie~$\mathscr C^\infty$
des formes sur~$X$ vers
\[ \sum_{IJ}  f^*(\beta_{IJ} \di y^I\wedge \dc y^J) 
 = f^*(\beta) .\]
Cela conclut la démonstration.
\end{proof}

\begin{coro}
La proposition précédente vaut si $\beta$ est une G-forme G-lisse
positive de bidegré~$(q,q)$ sur~$\R^n$
telle que le courant~$[\beta]$ soit fermé.
\end{coro}
\begin{proof}
En effet, par convolution, la G-forme $\beta$ est
limite d'une suite $(\beta_k)$ de formes lisses positives fermées.
Pour tout~$k$, on a $S\cap_f[\beta_k]=S\wedge f^*\beta_k$
et l'assertion voulue s'en déduit par passage à la limite,
compte tenu de la continuité de l'intersection tropicale
et du produit d'un courant positif fermé par une G-forme.
\end{proof}

\begin{prop}\label{prop.intertrop-assoc}
Soit $f\colon X\to\R^n$ une carte harmonique
et soit $p\colon \R^n\to\R^m$  une application affine; 
soit~$S$ un courant positif fermé sur~$X$,
soit~$T$ un courant positif fermé sur~$\R^n$ 
et soit~$U$ un courant positif fermé sur~$\R^m$.
On a 
\[ S \cap_f (T\cap_p U) = (S\cap_f T) \cap_{p\circ f} U. \]
\end{prop}
\begin{proof}
Soit $\omega$ une forme lisse à support compact sur~$X$ ;
démontrons que
\[ \langle S\cap_f (T\cap_p U) ,\omega\rangle = \langle
 (S\cap_f T)\cap_{p\circ f} U,\omega\rangle. \]
Par passage à la limite
(\S\ref{ss.lagerberg-convolution}), 
on peut supposer que $U$ est
le courant~$[\gamma]$ associé à une forme lisse~$\gamma$.
Alors $T\cap U=T\wedge f^*\gamma$ 
et $(S\cap_f T) \cap_{p\circ f}U=(S\cap_f T)\wedge (p\circ f)^*\gamma$.
Il s'agit donc de démontrer l'égalité
\[ \langle S\cap_f (T\wedge p^*\gamma) ,\omega\rangle = \langle
 (S\cap_f T)\wedge (p\circ f)^*\gamma ,\omega\rangle. \]
De nouveau par passage à la limite
(\S\ref{ss.lagerberg-convolution}), 
on peut alors supposer que $T$
est le courant~$[\beta]$ associé à une forme lisse~$\beta$;
on a alors $T\wedge p^*\gamma=[\beta\wedge p^*\gamma]$
et $S\cap_f T=S\wedge f^*\beta$, de sorte qu'il s'agit de prouver
l'égalité
\[ \langle S\wedge f^*(\beta\wedge p^*\gamma) ,\omega\rangle = \langle
 (S\wedge f^*\beta)\wedge (p\circ f)^*\gamma ,\omega\rangle, \]
laquelle est évidente.
\end{proof}

\begin{prop}\label{prop.intertrop-comm}
Soit $S$ et $T$ des courants positifs fermés sur~$\R^n$
et soit $f=\id_{\R^n}$.
On a $S\cap_f T=T\cap_f S$.
\end{prop}
\begin{proof}
Soit $\omega$ une forme lisse à support compact sur~$\R^n$;
démontrons que $\langle S\cap_f T,\omega\rangle=\langle T\cap_f S,\omega\rangle$.
Par passage à la limite
(\S\ref{ss.lagerberg-convolution}), il suffit de traiter le cas où il existe
une forme lisse~$\beta$ telle que $T=[\beta]$,
puis celui où il existe une forme lisse~$\alpha$
telle que $S=[\alpha]$.
Alors, 
\[ \langle S\cap_f T,\omega\rangle=\langle [\alpha]\wedge \beta,\omega\rangle
= \int \alpha \wedge \beta\wedge \omega \]
tandis que
\[ \langle T\cap_f S,\omega\rangle
= \int \beta \wedge \alpha\wedge \omega. \]
Puisque les degrés de $\alpha$ et $\beta$ sont pairs, ces deux expressions
coïncident.
\end{proof}

\begin{defi}
Soit $f\colon X\to\R^n$ une carte harmonique
et soit $T$ un courant positif fermé de bidegré~$(q,q)$
sur~$\R^n$.
On pose $f^!T=\delta_X \cap_f T$; c'est un courant positif fermé
de bidegré~$(q,q)$.
\end{defi}

\begin{prop}
Soit $f\colon X\to\R^n$ une carte harmonique
et soit $T$ un courant positif fermé de bidegré~$(q,q)$
sur~$\R^n$.
\begin{enumerate}
\item
Si $T=[\alpha]$ est le courant associé à une G-forme G-lisse~$\alpha$,
alors $f^!T=[f^*\alpha]$.

\item
Si $\mathscr T$ est un filtre de courants positifs fermés sur~$\R^n$
qui converge vers~$T$,
alors le filtre de courants $f^!\mathscr T$ sur~$X$ converge vers~$f^!T$.
\item
Si $f\colon\R^m\times\R^n\to\R^n$ est la seconde projection,
et si le calibrage de $\R^m$ est le calibrage constant évident,
alors $f^!T=\delta_{\R^m}\boxtimes T$.
\item
Si $g\colon \R^n\to\R^p$ est une application affine,
on a $(g\circ f)^!T=f^!(g^! T)$.
\end{enumerate}
\end{prop}
\begin{proof}
\begin{enumerate}
\item 
C'est la proposition~\ref{prop.intertrop-forme}.
\item
D'après la proposition~\ref{prop.hemicontinuite},
la filtre $\delta_X\boxtimes \mathscr T$ converge vers~$\delta_X\boxtimes T$.
L'assertion découle alors de la proposition~\ref{prop.BT-continu}.
\item
Par approximation, on peut supposer que $T$ est le
courant~$[\omega]$ associé à une forme lisse, positive, 
fermée~$\omega$ (\S\ref{ss.lagerberg-convolution}).
On a donc
$ f^!T=f^![\omega]=[f^*\omega]$.
Si $\alpha$ et~$\beta$ sont des formes lisses à support compact sur~$\R^m$
et~$\R^n$, on a donc, notant $e\colon \R^m\times\R^n\to\R^m$ la première projection, 
\[ \langle f^![\omega], e^*\alpha\wedge f^*\beta\rangle
= \int_{\R^m\times\R^n} e^*\alpha f^*(\omega\wedge\beta)
= \int_{\R^m} \alpha \int_{\R^n} \omega\wedge\beta
= \langle \delta_{\R^m},\alpha\rangle \langle
  [\omega ], \beta\rangle.\]
Compte tenu de la définition de $\delta_{\R^m}\boxtimes [\omega]$
(théorème~\ref{prop.ST-produit}), 
on a donc $f^![\omega]=\delta_{\R^m}\boxtimes[\omega]$.

\item
Par approximation, on peut supposer que $T$ est le courant~$[\omega]$
associé à une forme lisse positive fermée.
Dans ce cas, on a  
\[ (g\circ f)^![\omega]
=[(g\circ f)^*\omega]=[f^*(g^*\omega)]=f^![g^*\omega]=f^!(g^![\omega]).\]
\end{enumerate}
Cela conclut la démonstration.
\end{proof}

\section{Exemples}

Les exemples ci-dessous sont des variantes de ceux
donnés par Mihatsch dans~\cite{Mihatsch-2023}.
On note plus simplement $S\cap T$ l'intersection tropicale
$S\cap_{\id} T$ de deux courants positifs fermés~$S$ et~$T$ sur~$\R^n$.

\subsection{}
Soit $V$ et~$W$ deux sous-espaces vectoriels de~$\R^n$
tels que $V+W=\R^n$; on dit qu'ils sont transverses.
Cela revient à dire que l'on a
$\codim(V\cap W)=\codim(V)+\codim(W)$.
Soit $\mu_V$ un vecteur-volume sur~$V$,
soit $\mu_W$ un vecteur-volume sur~$W$.
Il existe un unique vecteur-volume $\mu_V\cap\mu_W$
sur~$V\cap W$ 
telle que la suite exacte usuelle d'espaces vectoriels
\[ 0 \to V\cap W \to V \oplus W \to \R^n \to 0 \]
soit compatible avec les vecteurs-volume $\mu_V\cap\mu_W$,
$\mu_V\otimes\mu_W$ et le vecteur-volume standard de~$\R^n$.

\begin{prop}
Soit $V$ et~$W$ deux sous-espaces vectoriels de~$\R^n$, transverses,
munis de vecteurs-volume~$\mu_V$ et~$\mu_W$.
On a l'égalité
\[ [V,\mu_V] \cap [W,\mu_W] = [V\cap W, \mu_V\cap \mu_W] \]
de courants sur~$\R^n$.
\end{prop}
\begin{proof}
Revenons à la définition.  
Le graphe de~$f=\id$ est la diagonale de~$\R^n\times\R^n$;
notons $j$ l'immersion diagonale de~$\R^n$ dans~$\R^{2n}$
et $p$ la première projection de~$\R^{2n}$ vers~$\R^n$.
On a 
\[ [V,\mu_V]\cap [W,\mu_W] = j^! ([V,\mu_V]\boxtimes[W,\mu_W])
 = p_* (([V,\mu_V]\boxtimes[W,\mu_W]) \cap \Gamma_f). \]
Le courant $[V,\mu_V]\boxtimes[W,\mu_W]$ sur~$\R^{2n}$
est le courant d'intégration sur $V\times W$ muni du calibrage
$\mu_V\otimes \mu_W$.
Il s'agit donc calculer $[V\times W,\mu_V\otimes\mu_W]\cap\Gamma_f$.
Appliqué par récurrence, le lemme suivant entraîne l'égalité
\[ [V\times W,\mu_V\otimes\mu_W]\cap\Gamma_f =
[V\times W\cap\Delta, \mu_V\otimes\mu_W \cap\mu_\Delta],\]
où $\mu_V\otimes\mu_W\cap\mu_\Delta$ est le calibrage de $V\times W\cap\Delta$
qui est compatible avec la suite exacte
\[ 0 \to V\times W\cap\Delta \to V\times W \xrightarrow{(y_i-x_i)} \R^n\to0.\]
La première projection identifie $V\times W\cap\Delta$ avec $V\cap W$
et le calibrage $\mu_V\otimes\mu_W\cap\mu_\Delta$ avec
le calibrage $\mu_V\cap\mu_W$.
Cela conclut la démonstration.
\end{proof}

\begin{lemm}
Soit $V$ un sous-espace vectoriel de~$R^n$ et soit $f\colon \R^n\to\R$
une forme linéaire non nulle.
Soit $W=V\cap\ker(f)$.
Si $V=W$, alors $[V,\mu_V] \wedge \ddc \max(f,0)=0$.
Sinon, $W$ est un hyperplan de~$V$,
il existe un unique calibrage~$\mu_W$ de~$W$
tel que la suite exacte $0\to W\to V\xrightarrow f\R\to0$
soit compatible avec les calibrages~$\mu_W$, $\mu_V$ et le calibrage
standard de~$\R$,
et l'on a $[V,\mu_V]\wedge\ddc \max(f,0)=[W,\mu_W]$.
\end{lemm}
\begin{proof}
Supposons d'abord que $V=W$. Lorsque $t\to 0$ par valeurs positives, la fonction
convexe~$\max(f-t,0)$ converge uniformément vers~$\max(f,0)$,
de sorte que $[V,\mu_V]\wedge \ddc\max(f,0)=\lim [V,\mu_V]\wedge\ddc\max(f-t,0)$.
Puisque $\max(f-t,0)\equiv0$ au voisinage du support~$V$ du courant~$[V,\mu_V]$,
le courant $[V,\mu_V]\wedge\ddc \max(f-t,0)$ est nul, pour tout $t>0$.
Ainsi, $[V,\mu_V]\wedge\ddc\max(f,0)=0$, comme il fallait démontrer.

Traitons maintenant le cas où $f|_V\neq0$, de sorte que $W$ est un hyperplan
de~$V$.
On choisit une base de~$\R^n$ de sorte que $V=\R{p+1}\times\{0\}$
et $f=x_{p+1}$, si bien que $W=\R^p\times\{0\}$.
Il existe un unique nombre réel~$a$ tel que $\mu_V=a\abs{e_1\wedge\dots\wedge e_{p+1}}$. Soit alors $\omega$ une forme lisse de bidegré~$(p,p)$
sur~$\R^n$, à support compact. Par définition, on a 
\[ 
\langle [V,\mu_V] \wedge \ddc \max(f,0),\omega\rangle
=
\langle \ddc( \max(f,0) [V,\mu_V] ),\omega\rangle
= \int_V   \langle x_{p+1}^+\ddc \omega ,\mu_V\rangle . \]
Notant $V_+$ le demi-espace $f^{-1}(\R_+)$, muni du calibrage~$\mu_V$,
de bord~$W$,
on a donc
\begin{align*}
  \langle [V,\mu_V] \wedge \ddc \max(f,0),\omega\rangle 
& = \int_{V^+} x_{p+1}\ddc\omega \\
&= \int_{V^+}^\partial  (x_{p+1}\wedge \dc\omega - \dc x_{p+1}\wedge\omega),
\end{align*}
compte tenu de la formule de Green et de ce que $\ddc x_{p+1}=0$.
Puisque $x_{p+1}\equiv0$ sur~$W$, on obtient
\[ 
  \langle [V,\mu_V] \wedge \ddc \max(f,0),\omega\rangle 
 = - \int_{V^+}^\partial \dc x_{p+1}\wedge\omega. \]
Par définition de l'intégrale de bord et du calibrage~$\mu_W$,
cette expression vaut $\int_W \omega$, comme il fallait démontrer.
\end{proof}

\subsection{}
Soit $P$ un sous-espace paralinéaire de~$\R^n$, purement de dimension~$p$,
muni d'un $p$-calibrage~$\mu_P$.
Alors, $(P,\mu_P)$ définit un courant squelettique de bidimension~$(p,p)$
sur~$\R^n$ que l'on note~$[P,\mu_P]$.
Si ce calibrage est positif, le courant $[P,\mu_P]$ est positif;
si ce calibrage est harmonieux, ce courant est fermé.

\subsection{}
Soit $P$ et $Q$ des sous-espaces paralinéaires de~$\R^n$, 
purement de dimensions~$p$ et~$q$ respectivement,
et respectivement munis d'un $p$-calibrage~$\mu_P$ et d'un $q$-calibrage~$\mu_Q$
que l'on suppose positifs et harmonieux.

On dispose ainsi de l'intersection tropicale des deux
courants $[P,\mu_P]$ et $[Q,\mu_Q]$. 
C'est un courant positif fermé de bidimension~$p+q-n$;
il est squelettique, c'est-à-dire qu'il correspond à 
un sous-espace paralinéaire~$R$ de~$\R^n$,
purement de dimension~$r=p+q-n$, muni d'un $r$-calibrage
positif et harmonieux.

Plusieurs constructions de cette intersection tropicale
ont été données dans la littérature. 
Expliquons pourquoi notre construction la retrouve.

Par construction, l'intersection tropicale se calcule localement.
Au voisinage d'un point $x\in P\cap Q$ appartenant aux intérieurs
de deux cellules de~$P$ et~$Q$ qui s'intersectent transversalement,
le calcul est celui de l'intersection tropicale de deux espaces linéaires.
Dans le cas général, on calculer l'intersection tropicale
transverse de $[P,\mu_P]$ avec un translaté « générique » $y+[Q,\mu_Q]$,
et faire tendre $y$ vers~$0$ : par continuité,
le résultat convergera vers $[P,\mu_P]\cap [Q,\mu_Q]$.
C'est la règle de déplacement des éventails.

\part{Géométrie analytique}

\chapter{Algèbre graduée et réduction des anneaux normés}

\section{Algèbre graduée}

\subsection{}
Sauf mention expresse du contraire, \textit{anneau gradué} signifiera dans ce texte
\textit{anneau $\R^\times_+$-gradué}, la graduation 
étant notée multiplicativement et en exposant. 
Nous mettrons parfois l'exposant entre crochets pour éviter une éventuelle
confusion avec les puissances.

Un morphisme d'anneaux gradués est un morphisme d'anneaux préservant les graduations. 
Si $A$ est un anneau gradué nous noterons $A^\times$ l'ensemble des éléments inversibles \textit{homogènes}
de $A$. La catégorie des anneaux ordinaires s'identifie à celle des anneaux $A$ trivialement gradués, c'est-à-dire tels que
$A^{[1]}=A$. 

Une part importante de l'algèbre commutative classique s'étend au cas gradué, la plupart du temps sans difficulté aucune : il suffit de 
rajouter un peu partout des hypothèses d'homogénéité. Nous allons en donner quelques exemples ici, le plus souvent sans démonstration. 
Pour un exposé plus complet, on pourra se référer au texte fondateur 
de Temkin \cite{temkin2004} ainsi
par exemple qu'aux paragraphes de rappels et compléments sur le sujet dans \cite{ducros2013}, \cite{ducros2021}
ou encore \cite{ducros-2024b}.

\subsection{}
Suivant Temkin, nous appellerons (improprement) \textit{corps gradué} tout anneau gradué non nul dans lequel tout
élément homogène non nul est inversible. Par exemple si $k$ est un corps (usuel) et $r$ un réel strictement positif, 
l'anneau des polynômes de Laurent $k[t,t^{-1}]$ muni de la graduation pour laquelle $at^i$ est homogène
de degré $r^i$ pour tout $a\in k^\times$ et tout $i\in \Z$ est un corps gradué. Une extension d'un corps gradué
$K$ est un corps gradué $F$ muni d'un morphisme d'anneaux gradués de $K$ dans $F$, nécessairement injectif.

\subsection{}
Si $A=\bigoplus_{r>0}A^r$ est un anneau gradué et si $a$ est un élément homogène de $A$, nous 
poserons $\|a\|=0$ si $a=0$ et désignerons sinon par $\|a\|$ l'unique réel strictement positif $r$ tel que $a\in A^r$. 
Si $K$ est un corps gradué, nous utiliserons de préférence
la notation $\abs a$ plutôt que $\|a\|$. (Nous noterons donc le degré comme une norme ou
une valeur absolue ; l'intérêt de ce choix apparaîtra plus loin lorsque
nous manipulerons des réductions graduées). 

\subsection{}\label{ss-matrices-graduees}
Si $A$ est un anneau gradué, un $A$-module gradué est un $A$-module $M$ muni en tant que groupe
abélien d'une décomposition en somme directe $M=\bigoplus_{r>0} M^r$ telle que $A^sM^r\subset M^{rs}$ pour tout $(r,s)$. 
Nous rencontrerons essentiellement cette notion lorsque $A$
est un corps gradué, auquel cas nous parlerons de
$A$-espace vectoriel gradué. 

Si $M$ et $N$ sont deux $A$-modules gradués, nous considérerons toujours $M\otimes_A N$ comme muni de la graduation 
pour laquelle $(M\otimes_A N)^r$ est le sous-groupe de $M\otimes_A N$ engendré par les éléments de la forme $m\otimes n$ 
avec $m\in M^s$, $n\in N^t$ et $st=r$. 
Ce produit tensoriel gradué représente le foncteur qui envoie un $A$-module gradué $P$
sur l'ensemble des applications bilinéaires $b$ de $M\times N$ vers $P$ telles que $b(M^s, N^t)\subset P^{st}$ pour
tout $(s,t)$. 

Une famille $(m_i)$ d'éléments homogènes d'un $A$-module gradué $M$ est libre (resp. génératrice, resp. une base) si l'application qui à
une famille $(a_i)$ d'éléments homogènes de $A$ à support fini associe $\sum a_i m_i$ est injective (resp. surjective, resp. bijective). 

On démontre comme en algèbre linéaire standard que tout espace vectoriel gradué possède une base et que toutes ses bases ont même
cardinal, que l'on appelle ici aussi la dimension.

On prendra garde qu'un $A$-module gradué admettant une base 
n'est pas en général isomorphe à $A^{(I)}$ pour un certain ensemble $I$, mais à une somme directe $\bigoplus_{i\in I} A[r_i]$ où les $r_i$ sont des réels strictement positifs et où $A[r_i]$ désigne le $A$-module gradué de
$A$-module
sous-jacent $A$ et dont la graduation 
est donnée par la formule $(A[r_i])^s=A^{r_i^{-1}s}$ (la notation du décalage est choisie de sorte que $1$ soit de degré $r_i$).  Si $r$ désigne la famille $(r_i)$ on écrira $A[r]$ au lieu de
$\bigoplus_i A[r_i]$. 

On peut développer un calcul matriciel dans ce contexte. La donnée d'un morphisme de 
$A$-modules gradués, c'est-à-dire d'une application $A$-linéaire préservant les graduations, 
de $\bigoplus_j A[s]\to \bigoplus_i A[r]$, où
$r=(r_i)$ et $s=(s_j)$, 
est équivalente à la donnée d'une matrice $(a_{ij})$ où les $a_{ij}$ sont 
des éléments homogènes de $A$, où les $a_{ij}$ à $j$ fixé sont presque tous nuls, et où $\|a_{ij}\|s_i\in \{r_j,0\}$ pour tout 
$i$ ; nous noterons $\mathrm M_{r,s}(A)$ le $A^1$-module constitué des matrices $(a_{ij})$ comme ci-dessus ; 
si $t$ est un troisième polyrayon, la formule usuelle du produit matriciel définit une application bilinéaire 
de $\mathrm M_{r,s}\times \mathrm M_{s,t}$ vers $\mathrm M_{r,t}$, qui correspond à la composition 
des morphismes de $A$-modules gradués. 
nous désignerons par $\GL_{s,r}(A)$ l'ensemble de ses éléments possédant un inverse bilatère dans
$\mathrm M_{r,s}(A)$. Si $r=s$
nous écrirons simplement  $\mathrm M_s(A)$ et  $\GL_s(A)$. 

\subsection{}
La catégorie des $A$-modules gradués admet un foncteur $\mathrm{Hom}$ interne : 
si $M$ et $N$ sont deux $A$-modules gradués, on note $\mathrm{Hom}_{\mathrm{gr}}
(M,N)$ la somme directe $\bigoplus_{\rho>}\mathrm{Hom}(M,N[\rho^{-1}])$ ; 
c'est de manière naturelle un $A$-module gradué (par $\rho$). 
Si $r=(r_1,\ldots, r_n)$ et $(s=(s_1,\ldots, s_m)$ sont deux
polyrayons, 
\[\mathrm{Hom}_{\mathrm{gr}}(A[s], A[r])\simeq \bigoplus_{i,j}A[r_j^{-1}s_i]\]
(considérer les coefficients matriciels). On notera
$M^\vee$ le $A$-module gradué
$\mathrm{Hom}_{\mathrm{gr}}(M,A)$ ; on l'appelle le dual de $M$. 

\subsection{}
Soit $A$ un anneau gradué. On dit que $A$ est intègre s'il l'est en tant qu'anneau standard ; cela revient à demander que tout
produit fini d'éléments \textit{homogènes} non nuls
de $A$ soit un corps.

On peut définir le localisé 
de $A$ relatif à
une partie multiplicative constituée d'éléments homogènes, et en particulier son corps gradué des fractions si $A$
est intègre. 

Le nilradical de $A$ (au sens standard) est un idéal homogène de $A$. On dit que l'anneau gradué $A$ est réduit s'il l'est au sens
classique, c'est-à-dire si son nilradical est nul. Cela revient à demander que tout élément nilpotent
homogène de $A$ soit nul. 

Un idéal homogène $I$ de  $A$ sera dit premier (resp. maximal) si et seulement si l'anneau gradué
$A/I$ est intègre (resp. est un corps gradué).

On dit que  $A$ est local s'il possède un et un seul idéal homogène maximal. 
Si $\mathfrak p$ est un idéal homogène premier de $A$, 
le localisé $A_{\mathfrak p}$ de $A$ relatif à la partie multiplicative de $A$ constitué des éléments homogènes
de $A\setminus\mathfrak p$ est un anneau gradué local, que nous appellerons le localisé
de $A$ en $\mathfrak p$. 

On dit que $A$ est noethérien (resp. artinien) si et seulement si toute famille non vide
d'idéaux \textit{homogènes} de $A$ admet un élément maximal (resp. un élément minimal). On démontre comme
dans le cas classique que si $A$ est noethérien il possède un nombre fini d'idéaux homogènes premiers minimaux, 
que tout idéal homogène premier  de $A$ contient l'un de ces idéaux premiers homogènes
minimaux, et que le localisé de $A$ en un idéal homogène premier minimal 
est artinien. 

On démontre également que $A$ est artinien si et seulement s'il est noethérien et si tous ses idéaux premiers
homogènes sont maximaux. Si $A$ est artinien il s'écrit comme produit d'un nombre fini d'anneaux gradués locaux artiniens. 
Si $A$ est un anneau gradué local artinien de corps gradué résiduel $K$,
il existe une suite finie $A=I_0\supseteq I_1\ldots \supseteq I_m=(0)$ d'idéaux
homogènes de $A$ tels que $I_j/I_{j+1}$ soit isomorphe à $K$ comme $A$-module pour tout $j\leq m-1$. 
L'entier $m$ ne dépend pas de la suite d'idéaux choisie et est appelé la longueur de $A$. On définit la longueur d'un anneau gradué
artinien non nécessairement local comme la somme des longueurs de ses facteurs locaux.

\subsection{}\label{subsec-polynomes-gradues}
Soit $A$ un anneau gradué et soit $r=(r_i)_{i\in I}$ une famille de réels strictement positifs. On note $A[T_i/r_i]$ ou
plus brièvement $A[T/r]$ la
$A$-algèbre graduée d'anneau sous-jacent $A[T_i]_{i\in I}$ et dont la graduation est telle que $A[T_i/r_i]^s$ soit 
l'ensemble des polynômes de la forme $\sum a_I T^I$ avec $a_I\in A^{r^{-I}s}$ pour tout $I$. Informellement, chaque
$T_i$ est une indéterminée homogène telle que $\|T_i\|=r_i$ ; rigoureusement, ce slogan se traduit par le fait que
pour toute $A$-algèbre graduée $B$ et toute famille $(b_i)$ d'éléments de $B$
tels que $b
_i\in B_{r_i}$ quel que soit $i$, il existe un unique morphisme de $A$-algèbres graduées envoyant $T_i$ sur $b_i$
pour tout $i$ (à savoir l'évaluation en $(b_i)$). 

On démontre comme dans le cas classique que si l'ensemble d'indices $I$ est fini et si $A$ est noethérien
alors $A[T/r]$ est noethérien. 

Ces algèbres de polynômes permettent en particulier de développer dans le contexte des corps gradués une théorie
de l'indépendance algébrique, des extensions algébriques, des bases et du degré de transcendance, analogue à celle bien connue
dans le cadre classique (les bases de transcendance seront toujours
par définition constituée d'éléments homogènes).

Soit $K$ un corps gradué et soit $(r_i)_{i\in I}$ une famille de réels strictement positifs.
Nous noterons $K(T_i/r_i)_i$, ou souvent plus simplement $k(T/r)$ le corps des fractions gradué
de $K[T_i/r_i]_i$. 
Le sous-groupe $\abs{K(T/r)^\times}$ de
$\R_+^\times$ est engendré par $\abs{K^\times}$ et les $r_i$, et la composante
$K(T/r)^1$ est l'ensemble des fractions de la forme $\sum a_I T^I/\sum b_I T^I$ où les $b_I$ sont non tous nuls 
et où les $\abs{a_I}r^I$ et $\abs{b_I}r^I$ non nuls sont tous égaux à un même réel strictement positif. La
famille $(T_i)$ est une base de transcendance de $K(T_i/r_i)_i$ sur $K$. 

Soit $L$ une extension quelconque de $K$ et soit $(t_i)$ une famille d'éléments
de $L^\times$.
La famille $(t_i)$ est algébriquement indépendante sur $k$
si et seulement s'il existe un $K$-isomorphisme de $K(T_i/\abs{t_i})_i$ sur $K(t_i)_i$ envoyant
$T_i$ sur $t_i$ pour tout $i$. 

\subsection{}\label{ss-tenseurk1-kr}
Soit $K$ un corps gradué et soit $A$ une $K$-algèbre graduée. 
Soit $r\in \abs{K^\times}$ et soit $\lambda$ un élément non nul de $K^r$. 
L'application naturelle de $A^1\otimes_{K^1}K^r$ vers $A^r$ est bijective, de
réciproque $a\mapsto (a/\lambda)\otimes \lambda$. 
Autrement dit, le morphisme de $K$-algèbres 
de $A^1\otimes_{K^1}K$ vers $A$ est une injection d'image $\bigoplus_{r\in \abs{K^\times}}A^r$.

\subsection{}\label{deg-extension-graduee}
Soit $K$ un corps gradué. 
Soit $L$ une extension de $K$ ; c'est en particulier un $K$-espace vectoriel,
dont la dimension sera appelée le degré de
$L$ sur $K$. On sait
par \ref{ss-tenseurk1-kr}
que le morphisme naturel de $L^1\otimes_{K^1}K$ vers $L$ est injectif et
identifie  $L^1\otimes_{K^1}K$ au sous-corps gradué $M=\bigoplus_{r\in \abs{K^\times}}L^r$ de $L$. Notons qu'on a par
construction $\abs{M^\times}=\abs{K^\times}$ et $L^1=M^1$. 
Nous nous proposons maintenant décrire le degré et le degré de transcendance de l'extension graduée $K\hookrightarrow L$ en termes 
d'invariants plus traditionnels.

Traitons d'abord le cas du degré.
Soit $(a_i)$ une base de $L^1$ sur $K_1$ (au sens classique) Puisque $M=L^1\otimes_{K^1}K$, la famille
$(a_i)$ est encore une base de $M$ sur $K$. 
Soit $(b_j)$ une famille d'éléments de $L^\times$ telle que les
$\abs{b_j}$ forment un système de représentants de $\abs{F^\times}/\abs{K^\times}$. Si $j_1$ et $j_2$
sont deux indices distincts et si $\mu_1$ et $\mu_2$ sont deux éléments de $M^\times$ alors
$\abs{\mu_1b_{j_1}}\neq \abs{\mu_2b_{j_2}}$ par choix des $b_j$ et parce que $\abs{M^\times}
=\abs{K^\times}$. La famille $(b_j)$ est donc libre sur $M$. Par ailleurs si $\lambda$ est un élément de 
$L^\times$ il existe par choix de $(b_j)$
un élément $a$ de $K^\times$ et un indice $j$ tel que $\abs \lambda=\abs{ ab_j}$. Le produit
$(\lambda/ab_j)a$ appartient alors à $(L^1)^\times\cdot K^\times\subset M^\times$, et
l'on a $\lambda =(\lambda/ab_j)ab_j$, ce qui montre que $(b_j)$ est une famille génératrice de 
$L$ sur $M$. C'en est donc une base
(et l'on a par conséquent $[L=M]=
(\abs{L^\times}:\abs{K^\times}$). 
Par conséquent $(a_ib_j)_{i,j}$ est une base de $L$ sur $K$
si bien que \[[L:K]=[L^1:K^1](\abs{L^\times}:\abs{K^\times}.)\]

Traitons maintenant le cas du degré de transcendance. Soit $(\alpha_i)$ une base de transcendance (au sens classique)
de $L^1$ sur $K^1$. Il est immédiat que les $\alpha_i$ sont encore algébriquement indépendants sur $K$. Et comme $L^1$
est algébrique sur $K^1(\alpha_i)_i$ et que $M=L^1\otimes_{K^1}K$, le corps gradué $M$ est algébrique sur 
$K(\alpha_i)_i$ ; la famille $(\alpha_i)$ est donc une base de transcendance de $M$ sur $K$. 
Soit $(\beta_j)$ une famille d'éléments de $L^\times$ tels que les $\abs{\beta_j}$ forment une base de
$\Q\otimes_\Z(\abs{F^\times}/\abs{K^\times})$ comme $\Q$-espace vectoriel. La liberté de la famille
$(\abs{\beta_j})$ et le fait que $\abs{L^\times}
=\abs{K^\times}$ entraînent que pour tout couple $(J_1,J_2)$ de multi-indices distincts et toute famille
$(\mu_1,\mu_2)$ d'éléments de $L^\times$ on a $\abs{\mu_1\beta^{J_1}}\neq \abs{\mu_2\beta^{J_2}}$ ; 
il en résulte que les $\beta_j$ sont algébriquement indépendants sur $M$. 
Par ailleurs si $\lambda$ est un élément de $L^\times$ le caractère générateur de $(\beta_j)$ entraîne
l'existence d'un entier non nul $N$, d'un élément $a$ de $K^\times$ et d'un multi-indice $J$
tels que $\abs{\lambda^N}=\abs{a\beta^J}$. Le produit
$(\lambda^N/a\beta^J)a$ appartient alors à $(L^1)^\times\cdot K^\times\subset M^\times$, et
l'on a $\lambda^N =(\lambda^N/a\beta^J)a\beta^J$, ce qui montre que $\lambda$ est algébrique 
sur $M(\beta_j)_j$. Il en résulte que $(\beta_j)$ est une base de transcendance de $L$ sur $M$. 
La concaténation de $(\alpha_i)$ et $(\beta_j)$ est alors une base de transcendance de $L$ sur $K$, si bien que 
\[\mathrm{degtr}(L/K)=\mathrm{degtr}(L^1/K^1)+\dim_\Q \Q\otimes_\Z(\abs{L^\times}/\abs{K^\times}).\]

\begin{lemm}\label{lem-transcendantpur-gradue}
Soit $K$ un corps gradué et soit $(r_i)_{i\in I}$ une famille de réels strictement positifs.
Il existe un sous-ensemble $J$ de $I$ et une
 famille $(S_j)_{j\in J}$ de monômes en les $T_i$ à coefficients dans $K^\times$ vérifiant
$\abs{S_j}=1$
pour tout $j$ telle que pour toute
extension graduée $L$ de $K$ vérifiant l'égalité $\abs{L^\times}=\abs{K^\times}$ les $S_i$ soient
algébriquement indépendants sur $L^1$ et engendrent $L(T/r)^1$ comme
extension de $L^1$. 
\end{lemm}

\begin{proof}
On munit l'ensemble d'indices $I$ d'un bon ordre. Nous allons construire par récurrence sur 
l'élément $\ell$ de $I$ 
un ensemble $B_\ell$ de monômes en les $T_i$ pour $i\leq \ell$ tels que les conditions suivantes soient satisfaites :
\begin{enumerate}
\item pour tout $m\leq \ell$ l'ensemble $B_m$ est égal au sous-ensemble de $B_\ell$ constitué des monômes
ne faisant intervenir que les $T_i$ pour $i\leq m$, et $B_\ell\setminus \bigcup_{m<\ell}B_m$
a au plus un élément ; 
\item pour tout $\ell$ et toute
extension graduée $L$ de $K$ vérifiant l'égalité $\abs{L^\times}=\abs{K^\times}$ les éléments de $B_\ell$
sont algébriquement indépendants sur $L^1$ et engendrent $[L(T_i/r_i)_{i\leq \ell}]^1$ comme
extension de $L^1$. 
\end{enumerate}
La réunion des $B_\ell$ tels que 
$B_\ell\setminus \bigcup_{m<\ell}B_m$
soit non vide (et donc un singleton)
fournira alors une famille $(S_j)$ comme
requise dans l'énoncé. 

Il suffit de fabriquer $B_\ell$ en supposant que les $B_i$ pour $i<\ell$ ont été construits. 
Posons $B=\bigcup_{i<\ell}B_i$. Soit
$L$ une extension de $K$ telle que $\abs{F^\times}=\abs{K^\times}$. Posons
$\breve K=K(T_i/r_i)_{i<\ell}$
et $\breve L=L(T_i/r_i)_{i<\ell}$ ; on a alors $\abs{\breve L^\times}=\abs{\breve K^\times}$. 
Par construction, 
les éléments de $B$ appartiennent à $\breve L^1$, sont algébriquement indépendants sur $L^1$
et engendrent $\breve L^1$ comme extension de $L^1$. 
Supposons tout d'abord que $r_\ell$ n'est pas de torsion modulo $\abs{\breve K^\times}$. 
Il ne l'est alors pas modulo $\abs{\breve L^\times}$, ce qui entraîne que les composantes homogènes
de $\breve L[T_\ell/r_\ell]$ sont les $\breve L^r\cdot T_\ell^n$ où $r$ parcourt $\abs{\breve L^\times}$ et où $n$
parcourt $\N$, et celles de $\breve L(T_\ell/r_\ell)$ sont par conséquent les  $\breve L^r\cdot T_\ell^n$ où $r$
parcourt $\abs{\breve L^\times}$ et où $n$
parcourt $\Z$. En particulier $\breve L(T_\ell/r_\ell)^1$ est égal à $\breve L^1$, ce qui montre que  $ L(T_\ell/r_\ell)^1$
est engendré par $B$ sur son sous-corps $L^1$. On peut donc poser $B_\ell=B$. 
Supposons maintenant que $r_\ell$ est de torsion modulo $\abs{\breve K^\times}$, et soit $n$ son ordre 
modulo $\abs{\breve K^\times}$. Le groupe $\abs{\breve K^\times}$ étant engendré par $\abs{K^\times}$ et les
$r_i$ pour $i<\ell$ il existe un monôme $a$ dont le coefficient appartient à $K^\times$ et 
qui ne met en
jeu que des $T_i$ avec $i<\ell$ tel que $r_\ell^n=\abs a$. L'élément 
$T_\ell^n/a$ de $\breve K(T_\ell/r_\ell)$ appartient alors à $\breve K(T_\ell/r_\ell)^1\subset \breve L(T_\ell/r_\ell)^1$. 
Il est transcendant sur $\breve L^1$ puisque $T_\ell$ est transcendant sur $\breve L$ par construction. Il suffit maintenant de vérifier
que $\breve L(T_\ell/r_\ell)^1=\breve L^1(T_\ell^n/a)$ ; cela assurera qu'on peut prendre $B_\ell=B\cup\{T_\ell^n/a\}$ et
terminera la démonstration. 

Donnons-nous un élément
$\lambda=(\sum_p a_p T_\ell^p)(\sum b_p T_\ell^p)$
de $\breve L(T_\ell/r_\ell)^1$, où les $a_m$ et
les $b_p$ sont des éléments
homogènes de $\breve L$, où les $b_p$ sont  non tous nuls, et où les
$\abs{a_p}r_\ell^p$ et les $\abs{b_p}r_\ell^p$ non nuls  sont tous égaux. Puisque $r_\ell$ est
d'ordre $n$
modulo $\abs{\breve L^\times}$, les entiers $p$ tels que $a_p\neq 0$ et les entiers $p$
tels que $b_p\neq 0$ appartiennent tous à la même
classe modulo $n$. Il existe
donc un entier $d\in\{0,\ldots, n-1\}$ tel que $\lambda$
puisse s'écrire $(T^d\sum \alpha_pT^{np})/(T^d\sum \beta_p T^{np})$,
où les $\alpha_p$ et
les $\beta_p$ sont des éléments
homogènes de $\breve L$, où les $\beta_p$ sont  non tous nuls,  et où les
$\abs{a_p}r_\ell^{n^p}$ et
les $\abs{b_p}r_\ell^{np}$ non nuls  sont tous égaux. On a ainsi
\[ \lambda=\frac {\sum (\alpha_pa^p)(T^n/a)^p} {\sum (\beta_pa^p)(T^n/a)^p},\]
et les $\abs{\alpha_pa^p}$ et $\abs{\beta_pa^p}$ non nuls sont tous égaux à un même $\abs{c}$ pour un 
certain $c\in \breve L^\times$. L'égalité \[ \lambda=\frac {\sum (\alpha_pa^p/c)(T^n/a)^p} {\sum (\beta_pa^p/c)(T^n/a)^p}\]
montre
alors que $\lambda$ appartient à $\breve L^1(T_\ell^n/a)$, ce qu'on souhaitait établir. 
\end{proof}

\begin{coro}\label{coro-gradue-tf}
Soit $K$ un corps gradué et soit $L$ une extension graduée de type fini de $K$. 
L'extension $L^1$ de $K^1$ et le groupe abélien $\abs{L^\times}/\abs{K^\times}$ sont de
type fini.
\end{coro}
\begin{proof}
Par hypothèse, $L$ s'identifie à une extension finie de $K(T/r)$ pour un certain polyrayon $r=(r_1,\ldots, r_n)$
de longueur finie. Le groupe $\abs {K(T/r)^\times}/\abs{K^\times}$ est engendré par les $r_i$ et est donc
de type fini, et $K(T/r)^1$ est de type fini (et de plus transcendant pur) sur $K^1$ par le lemme précédent. 
Comme $[L^1:K(T/r)^1]\left(\abs{L^\times}/\abs{K(T/r)^\times}\right)=[L:K(T/r)]$ 
en vertu de~\ref{deg-extension-graduee}, $L^1$ est de type fini sur $K^1$ et 
$\abs{L^\times}/\abs{K^\times}$ est de type fini. \end{proof}

\subsection{}
Soit $K$ un corps gradué. Une valuation (graduée) sur $K$ est une application 
$v  \colon K^\times\to G$, où $G$ est un groupe
abélien ordonné noté multiplicativement, que l'on prolonge systématiquement en adjoignant un plus petit
élément absorbant $0$ à $G$ et en posant $v(0)=0$, et qui satisfait les axiomes suivants : 
\begin{itemize}
\item pour tout couple
$(\lambda, \mu)$ d'éléments de $K^\times \cup\{0\}$ on a
$v(\lambda\mu)=v(\lambda)v(\mu)$ ; 
\item pour tout $r>0$ et tout couple $(\lambda, \mu)$ d'éléments
de~$K^r$,
on a $v(\lambda+\mu)\leq \max(v(\lambda, v(\mu))$. 
\end{itemize}

On dit que deux valuations $v$ et $v'$ sur $K$, à valeurs respectives dans
les groupes abéliens ordonnés $G$ et $G'$, sont équivalentes s'il existe un groupe abélien ordonné
$G''$, une valuation $v''$ sur $K$ à valeurs dans $G''$, et deux morphismes injectifs croissants $i\colon G''\to G$
et $j\colon G''\to G'$ tels que $v=i\circ v''$ et $v'=j\circ v''$. 

À une valuation $v$ sur $K$ est associé un sous-anneau gradué $\mathscr O_v$ de $K$, caractérisé par le fait que
$(\mathscr O_v)^r=\{\lambda \in K^r, v(\lambda)\leq 1\}$ pour tout $r>0$.  Pour tout $\lambda\in K^\times$ on a
$\lambda\in \mathscr O_v$ ou $\lambda^{-1}\in \mathscr O_v$ (en particulier, $K$ est le corps gradué des fractions
de $\mathscr O_v$). L'anneau gradué $\mathscr O_v$ est  local, au sens où il possède un et un seul idéal 
homogène maximal $\mathfrak m$,  caractérisé par le fait que 
$\mathfrak m^{[r]}=\{\lambda \in K^r, v(\lambda)<1\}$
pour tout $r>0$. 
Réciproquement si $A$ est un sous-anneau gradué de $K$ tel que 
$\lambda\in A$ ou $\lambda^{-1}\in A$ pour tout  $\lambda\in K^\times$ alors $K^\times/A^\times$ 
possède un ordre compatible avec sa structure de groupe dont l'ensemble des
éléments $\leq 1$ est l'image de $A\cap K^\times$, et l'application quotient $v$ de $K^\times$ vers $K^\times/A^\times$
est, relativement à l'ordre en question, une valuation sur $K$ telle que $\mathscr O_v=A$. On vérifie comme dans le cas classique
que ces constructions mettent en bijection l'ensemble des classes d'équivalence de valuations sur $K$ et celui des sous-anneaux gradués
$A$ de $k$ tels que $\lambda\in A$ ou $\lambda^{-1}\in A$ pour tout  $\lambda\in K^\times$. 
Un tel sous-anneau gradué de $K$ est appelé un anneau gradué de valuation (de $K$). 

Si $F$ est une extension graduée de $K$, toute valuation $v$ sur $K$ admet un prolongement~$w$ à~$F$ :
comme dans
le cas classique, on peut par exemple construire $\mathscr O_w$ 
comme un sous-anneau gradué de $F$ dominant 
$\mathscr O_v$ et maximal pour la relation de domination.

\section{Schémas et fibrés vectoriels gradués}

\subsection{}
Soit $A$ un anneau gradué. Nous nous permettrons (lorsqu'il sera clair que l'on 
travaille dans le contexte gradué) de noter $\Spec A$ l'ensemble des idéaux
premiers \textit{homogènes}
de $A$, équipé de la topologie de Zariski, dont les fermés
sont les parties de la forme $\{\mathfrak p, I\subset \mathfrak p\}$ 
où $I$ est un certain idéal homogène de $A$ ; cet espace est muni d'un faisceau
en anneaux gradués qui envoie tout ouvert $D(f):=\{\mathfrak p, f\notin
\mathfrak p\}$ (où $f$ est un élément
homogène de $A$) sur le localisé gradué $A_f$ ; les fibres de ce faisceau
sont des anneaux gradués locaux, ce qui fait de $\Spec A$ un
\textit{espace localement annelé gradué}. 
Nous appellerons schéma gradué tout espace localement annelé gradué localement
isomorphe à un spectre d'anneau gradué. Topologiquement, tout schéma gradué est un espace spectral. 
Tout point $x$ d'un schéma gradué possède un corps gradué
résiduel $\kappa(x)$. Un schéma gradué est dit réduit si ses anneaux gradués locaux sont réduits, et noethérien
s'il possède un recouvrement ouvert fini par des spectres d'anneaux gradués noethériens. Un schéma gradué noethérien $X$
possède un ensemble fini de composantes irréductibles $X_i$. Chacune d'elle a un unique point générique $\eta_i$, 
l'anneau local gradué $\mathscr O_{X_i,\eta_i}$ est artinien, et sa longueur est appelée la multiplicité générique de $X$
le long de $X_i$. 

Un morphisme de schémas gradués est un morphisme d'espaces localement annelés
gradués. On peut définir dans ce contexte, en décalquant ce qui se
fait dans le cas classique, les notions de morphisme fini, d'immersions
(ouvertes, fermées, puis générales), de produits fibrés, de changement de base, etc.

Lorsque $A$ est trivialement gradué, son spectre au sens ci-dessus
coïncide avec son spectre au sens usuel (son faisceau structural est 
donc également trivialement gradué). Les schémas usuels sont donc des cas
particuliers de schémas gradués. 

Si $X$ est un schéma gradué sur un corps gradué $K$ et si $L$ est une extension graduée de $K$, nous noterons souvent~$X_L$
le $L$-schéma gradué $X\times_K L$. Le morphisme naturel de~$X_L$ vers~$X$ 
est surjectif : en effet, sa fibre en un point $x$ de
$X$ est le spectre de l'anneau gradué non nul $\kappa(x)\otimes_K L$.

\subsection{Espaces affines et projectifs gradués}\label{def-proj-gradue}
Soit 
$A$ un anneau gradué. 
Soit $r=(r_1,\ldots, r_n)$ un polyrayon. 
On lui associe un schéma gradué affine de type fini sur $A$, l'espace affine relatif de polyrayon $r$, 
noté $\A^r_A$ et défini comme étant le spectre de $\Spec A[T/r]$. 

Soit $\rho=(\rho_0,\ldots, \rho_n)$ un polyrayon.
On lui associe un $A$-schéma gradué de type fini sur $A$, 
l'espace projectif de relatif de polyrayon $\rho$, noté $\P^\rho_A$ et
qu'on peut définir de deux façons 
différentes : 

\begin{itemize}
\item par une variante graduée de la construction $\mathrm{Proj}$, appliquée
à $A[T_0/s_0,\ldots, T_n/s_n]$ (dont on doit alors considérer non
seulement la
$\R_+^\times$-graduation mais aussi la graduation monomiale) ; 
\item par recollement des $U_j:=\Spec A[\tau_{ij}/(s_i/s_j)]_{i\neq j}$
le long des isomorphismes $D(\tau_{kj})\simeq D(\tau_{jk}), 
\tau_{ij}\mapsto \tau_{ik}\tau_{jk}^{-1}$. 
\end{itemize}
L'ouvert $U_j$ de la seconde construction correspond à l'ouvert $D(T_j)$ de la
première, avec $\tau_{ij}=T_i/T_j$. 

\subsection{Propreté et projectivité}
On peut définir un morphisme propre entre schémas gradués exactement 
comme dans le cas classique : un morphisme $f$ est propre s'il est 
de type fini et satisfait
le critère valuatif (il faut bien sûr prendre ici ce critère au sens
des valuations graduées) ; cela revient à demander que $f$ soit de type fini, 
séparé et universellement fermé. La preuve est analogue à celle du cas classique, 
les deux points clef étant les suivants : 

\begin{itemize}
\item tout anneau local gradué est dominé par un anneau gradué de valuation 
(on en déduit que toute spécialisation $y\to x$ sur un schéma gradué $X$
est induite par un moorphisme $\Spec A\to X$ où $A$ est un anneau de valuation 
gradué) ; 
\item l'image d'un morphisme quasi-compact $f\colon  Y\to X$ est fermée
si et seulement si elle est stable par spécialisation. 
\end{itemize}
La première assertion se démontre aisément avec le lemme de Zorn. 
Pour la seconde, l'implication directe est claire. 
Pour l'implication réciproque, on se ramène au cas d'un morphisme $f\colon \Spec B\to \Spec A$ ; 
en quotientant $A$
par le noyau de $f^*$, on peut supposer que $A$
s'injecte dans $B$. Dans ce cas pour tout idéal homogène premier minimal
$\mathfrak p$ de $A$ le produit tensoriel
$B\otimes_A A_{\mathfrak p}$ est non nul, ce qui entraîne que $\mathfrak p$ 
appartient à l'image de $f$ ; comme celle-ci est stable par spécialisation, et
comme tout idéal premier homogène de $A$ contient par le lemme de Zorn un 
idéal premier homogène minimal, $f$ est surjective. 

Si $r=(r_0,\ldots, r_n)$ est un polyrayon et si $A$ est un anneau gradué, 
le $A$-schéma gradué $\P^r_A$ est propre -- on montre en
chassant les dénominateurs
qu'il satisfait au critère valuatif. En conséquence tout $A$-schéma
gradué projectif
(c'est-à-dire se factorisant par une immersion fermée dans $\P^r_A$ pour un
certain $r$) est propre

\subsection{}\label{ss.fibrevect-gradue}
Soit $(X,\mathscr O_X)$
un espace localement annelé gradué. Un fibré vectoriel (gradué) sur $X$
est un $\mathscr O_X$-module gradué
$E$ tel que pour tout $x\in X$ il existe un
voisinage ouvert $U$ de $x$, un polyrayon $r=(r_1,\ldots, r_n)$ et
un isomorphisme $E|_U\simeq \mathscr O_X[r]$. L'entier $n$ est alors
uniquement déterminé par $x$ et est appelé rang de $
E$ en $x$ ; ce rang est une fonction localement constante sur $X$. 
Un fibré en droites sur $X$ est un fibré vectoriel de rang partout
égal à $1$.

Soit $E$ et $F$ deux fibrés vectoriels sur $X$. 
Le faisceau $\underline{\mathrm{Hom}}_{\mathrm{gr}}(E,F)$ est défini
comme
la somme directe des faisceaux en $\mathscr O_X^{[1]}$-modules 
$\underline{\mathrm{Hom}}(E,F[\rho^{-1}])$,
pour $\rho\in\R_+^\times$.
C'est de manière naturelle un faisceau en $\mathscr O_X$-modules gradués. 
C'est même un fibré vectoriel : en effet,
si $U$ est un ouvert de $X$ tel que $E|_U\simeq \mathscr O_U[s]$
et $F|_U\simeq \mathscr O_U[r]$
avec $r=(r_1,\ldots, r_n)$ et $
s=(s_1,\ldots, s_m)$ alors $\underline{\mathrm{Hom}}_{\mathrm{gr}}(E,F)|U$
est isomorphe à $\bigoplus_{i,j} \mathscr O_U[r_j^{-1}s_i]$ (considérer
les coefficients matriciels). 

Le fibré vectoriel $\underline{\mathrm{Hom}_{\mathrm{gr}}}(E,\mathscr O_X)$
sera noté $E^\vee$ et est appelé le dual de $E$. 

On dit qu'un fibré en droites $E$ sur $X$ est engendré par ses sections
globales si pour tout $x\in X$ il existe une section globale homogène
de $E$
constituant une base de $E$ au-dessus d'un voisinage de $x$. 

Supposons que $X$ est un schéma gradué sur un corps gradué $K$, 
que $L$ est une extension graduée de $K$ et que $D$ est un fibré en droites
sur $X$. Si l'image réciproque $D_L$ de $D$ sur $X_L$ est engendrée par ses
sections globales, alors $D$ est engendré par ses sections globales. 
Il suffit en effet de remarquer (en raisonnant sur un recouvrement affine)
que $D_L(X_L)=D(X)\otimes_k L$, et que si $(\sum_j s_{ij}\otimes \ell_{ij})_i$
est une famille de sections globales (homogènes) de $D_L$ engendrant ce dernier, 
alors les $s_{ij}$ engendrent $D$ 
(cela vaudrait plus généralement si $K$ était un anneau gradué et $L$
une $K$-algèbre graduée fidèlement plate). 

\subsection{}
Soit $A$ un anneau gradué et soit $\rho=(\rho_0,\ldots, \rho_n)$ un polyrayon. 
Le schéma projectif gradué $\P^\rho_A$ possède un fibré en droites naturel
noté $\mathscr O(1)$. En reprenant les notations
de \ref{def-proj-gradue}, le $A$-module gradué
$\mathscr O(1)(\P^\rho_A)$ est le sous-$A$-module
gradué de $A[T_0,\ldots, T_n]$ engendré par les $T_j$. 
Pour tout $j$, la restriction de 
$\mathscr O(1)$ à $U_j$ est libre de base $T_j$,
si bien que $\mathscr O(1)|_{U_j}\simeq \mathscr O_{U_j}[r_j]$ ; 
sur $U_j\cap U_k$ on a l'égalité $T_j=\tau_{jk}T_k$ entre sections 
homogènes de $\mathscr  O(1)$. 

Soit $X$ un $A$-schéma gradué et soit $\rho=(\rho_0,\ldots, \rho_n)$ un polyrayon.
Se donner un $A$-morphisme $f$ de $X$ vers $\P^\rho_A$ revient à se donner
un fibré en droites $D$ sur $X$ et des sections globales $s_0,\ldots, s_n$
de $D$, chaque $s_i$ étant homogène de degré $\rho_i$, qui ne s'annulent pas
simultanément sur $X$ (le morphisme $f$
étant donné, prendre $D=f^*\mathscr O(1)$ et $s_i=f^*T_i$ ; le fibré $D$ et
les $s_i$ étant données, $f$ est caractérisé par les
conditions $f^{-1}(U_j)=D(s_j)$
et $f^*(T_i/T_j)=s_i/s_j$). 

On dit qu'un fibré en droites $D$ sur $X$ est très ample s'il existe
$\rho=(\rho_0,\ldots, \rho_n)$ et 
une immersion $f\colon X\to \P^\rho_A$ telle que $D\simeq f^*\mathscr O(1)$, 
et qu'il
est ample s'il existe $m\geq 0$ tel que $D^{\otimes m}$ soit très ample.

\section{Schémas et fibrés vectoriels gradués stricts}

\begin{lemm}\label{lemm-tenseur-K1K}
Soit $K$ un corps gradué et soit $A$ une $K^1$-algèbre au sens classique. 

\begin{enumerate}
\item L'application $I\mapsto I\cap A$ établit un isomorphisme d'ensembles
ordonnés entre l'ensemble des idéaux homogènes de 
$A\otimes_{K^1}K$ et celui des idéaux de $A$, de réciproque 
$J\mapsto J\otimes_{K^^1}K$. 

\item L'anneau $A$ est local (resp. noethérien, resp. artinien) si et seulement si l'anneau gradué
$A\otimes_{K_1}K$ est local (resp. noethérien, resp. artinien). 

\item L'anneau $A$ est intègre (resp. réduit, resp. est un corps) si et seulement si l'anneau gradué
$A\otimes_{K^1}K$ est intègre (resp. réduit, resp. est un corps gradué).

\item Un idéal $J$ de $A$ est premier (resp. maximal) si et seulement si $J\otimes_{K^1}K$
est premier (resp. maximal). 

\item Supposons $A$ noethérien et soient $\mathfrak p_1,\ldots, \mathfrak p_n$ ses
idéaux premiers minimaux. Les $\mathfrak p_i\otimes_{K_1}K$ sont les idéaux homogènes
premiers minimaux de $A\otimes_{K^1}K$. Pour tout $i$ le localisé de  $A\otimes_{K^1}K$
en $\mathfrak p_i\otimes_{K^1}K$ s'identifie à $A_{\mathfrak p_i}\otimes_{K^1}K$, a pour corps gradué
résiduel $\mathrm{Frac}(A/\mathfrak p_i)\otimes_{K^1}K$ et a même longueur que $A_{\mathfrak p_i}$. 

\item Soit $d$ un entier et soit $M$ un $A$-module. Le $A^1$-module $M$ est libre de rang $d$ si
et seulement si le $A$-module $M\otimes_{A^1}A$ est libre de rang $d$. 

\end{enumerate}

\end{lemm}

\begin{proof}
Commençons par remarquer que $A\otimes_{K^^1}K=\bigoplus_{r\in \abs{K^\times}}A\otimes_{K^1}K^r$. 
Si $x$ est un élément homogène non nul de $A$ il existe donc $y\in K^\times$ tel que
$\abs y=\abs x$; on a alors $x=y(y^{-1}x)$, et l'on remarque que $\abs{y^{-1}x}=1$, c'est-à-dire 
que $xy^{-1}x\in A$. 

Si $J$ est un idéal de $A$ alors $J\otimes_{K^1}K$ est l'idéal homogène
$\bigoplus_{r\in \abs{K^\times}}J\otimes_{K^1}K^r$
de l'anneau gradué $A\otimes_{K^1}K=\bigoplus_{r\in \abs{K^\times}}A\otimes_{K^1}K^r$. 
On a donc $(J\otimes_{K^1}K)\cap A=J$. 
Soit maintenant $I$ un idéal homogène de $A\otimes_{K^1}K$. 
L'injection naturelle de $(I\cap A)\otimes_{K^1}K$ vers $A$ prend ses valeurs dans $I$, et 
a pour image $I$ tout entier. Il suffit en effet de vérifier que si $x$ est un élément homogène non nul de $I$, 
il appartient à $(I\cap A)\otimes_{K^1}K$, ce que l'on voit en 
écrivant $x=y(y^{-1}x)$ comme en début de preuve. 
Ceci achève de démontrer (a), et (b) s'en déduit puisque les trois assertions en jeu
peuvent s'exprimer en termes des ensembles ordonnés des idéaux de $A$ et des idéaux homogènes
de $A\otimes_{K^1}K$. 

Il est clair que si $A\otimes_{K^1}K$ est intègre (resp. réduit, resp. intègre) il en va de même de $A$
qui est sa composante homogène de degré $1$. Pour la réciproque, donnons-nous 
deux éléments homogènes non nuls $x$ et $x'$ de $A\otimes_{K^1}K$. 
Par la remarque faite en début de preuve, on peut écrire $x=yz$ et $x'=y'z'$ où $z$ et $z'$ sont des éléments non nuls de 
$A$ et où $y$ et $y'$ appartiennent à $K^\times$. Si $A$ est intègre le produit $xx'=yy'zz'$ est alors non nul, si bien que 
$A\otimes_{K^1}K$ est intègre. Et si $A$ est réduit (resp. un corps) alors $x=yz$ n'est pas nilpotent (resp. est inversible) si bien que 
$A\otimes_{K^1}K$ est  réduit (resp. est un corps gradué). 

L'assertion (d) se montre en appliquant (c) à $A/J$. 

Montrons (e). 
Il résulte de (a) et (d) que les $\mathfrak p_i\otimes_{K^1}K$ sont les idéaux homogènes premiers minimaux de 
$A\otimes_{K^1}K$. Fixons l'indice $i$ et notons $\ell$
la longueur de $A_{\mathfrak p_i}$. 
La commutation du produit tensoriel aux quotients et localisations (tout ceci dans le contexte gradué) assure que
le localisé de  $A\otimes_{K^1}K$
en $\mathfrak p_i\otimes_{K^1}K$ s'identifie à $A_{\mathfrak p_i}\otimes_{K^1}K$
et a pour corps gradué
résiduel $\mathrm{Frac}(A/\mathfrak p_i)\otimes_{K^1}K$. Il reste à vérifier
que la longueur de 
$A_{\mathfrak p_i}\otimes_{K^1}K$ est égale à
$\ell$. 
Choisissons une suite
d'idéaux $A_{\mathfrak p_i}=I_0\supseteq\ldots \supseteq I_\ell=(0)$ d'idéaux de $A_{\mathfrak p_i}$
telle que $I_j/I_{j+1}$ soit isomorphe à $\mathrm{Frac}(A_{\mathfrak p_i})$ comme $A_{\mathfrak p_i}$-module
pour tout $j<\ell$. Alors $(I_j\otimes_{K^1}K)/(I_{j+1}\otimes_{K^1}K)$
est isomorphe à $\mathrm{Frac}(A_{\mathfrak p_i})\otimes_{K^1}K$ comme $A_{\mathfrak p_i}\otimes_{K^1}K$-module
pour tout $j<\ell$, et la suite d'idéaux homogènes \[A_{\mathfrak p_i}\otimes_{K^1}K=
I_0\otimes_{K^1}K\supseteq\ldots I_j\otimes_{K^1}K\supseteq\ldots \supseteq I_\ell\otimes_{K^1}K=(0)\]
atteste ainsi
que l'anneau gradué local artinien $A_{\mathfrak p_i}\otimes_{K^1}K$ est de longueur $\ell$. 

Prouvons enfin \emph f). L'implication directe est évidente. Pour la réciproque on peut invoquer la descente fidèlement
plate graduée (qui se démontre comme sa variante traditionnelle), mais on peut aussi le voir beaucoup plus
directement comme suit. 
Soit $(m_1,\ldots, m_d)$ une base de $M\otimes_{A^1}A$. Pour tout $i$, 
l'élément $m_i$ de $M\otimes_{A^1}A$ est homogène de degré $\abs{\lambda_i}$ pour un certain $\lambda_i\in K^\times$.
Les $\lambda_i^{-1}m_i$ forment encore une base
de $M\otimes_{A^1}A$ sur $A$ et comme ils sont de degré $1$, ils forment une base de $M$ sur $A^1$. 
\end{proof}

\subsection{}
Soit $K$ un corps gradué. Soit $A$ une $K$-algèbre
graduée. 
Le morphisme naturel de $A^1\otimes_{K^1}K$ vers $A$
identifie $A^1\otimes_{K^1}K$ à la sous-algèbre
$\bigoplus_{r\in \abs{K^\times}}A^r$ de $A$. 
Nous dirons que la $K$-algèbre graduée $A$ est stricte si 
$A=A^1\otimes_{K^1}K$, c'est-à-dire donc si $\abs a\in \abs{K^\times}$
pour tout élément homogène non nul $a$ de $A$. 

Supposons que ce soit le cas. Le morphisme naturel 
de $\Spec A$ vers $\Spec A^1$ est alors topologiquement un homéomorphisme 
(lemme \ref{lemm-tenseur-K1K}, (a) et (d)), et il résulte des constructions que
$\mathscr O_{\Spec A}=\mathscr O_{\Spec A^1}\otimes_{K^1}K$. 

Nous dirons qu'un $K$-schéma est strict s'il est recouvert par des ouverts
affines de la forme $\Spec A$, où $A$ est une $A$-algèbre stricte. 
Tout ouvert affine d'un tel schéma est alors de cette forme. 

Si $Y$ est un $K^1$-schéma nous noterons $Y_K$ 
le $K$-schéma gradué $Y\times_{K^1}K$, qui s'identifie à 
 l'espace localement
 annelé gradué
 $(\abs Y, \mathscr O_Y\otimes_{K^1}K)$ (où $\abs Y$ désigne l'espace
 topologique sous-jacent à $Y$) ; 
si 
 $X$ est un $K$-schéma gradué nous noterons~$X^1$ l'espace localement annelé
~$(\abs X, \mathscr O_X^1)$. 
Les foncteurs
$Y\mapsto Y_K$ et $X\mapsto X^1$ 
établissent une équivalence entre la catégorie des $K^1$-schémas et
celle des $K$-schémas gradués stricts. 
Ces foncteurs préservent le caractère noethérien, irréductible, 
réduit, de type fini ; ils préservent plus généralement les composantes irréductibles et leurs 
multiplicités génériques (dans le cas noethérien); ces affirmations se déduisent du lemme 
\ref{lemm-tenseur-K1K}. 

Soit $X$ un $K$-schéma gradué strict. Il est propre si et seulement si $X^1$ est un $K^1$-schéma
propre. En effet, si $X^1$ est propre sur $K^1$ alors $X$ est propre sur $K$
cas $X=X\times_{K^1}K$. Et si $X$ est propre sur $K$, alors $X^1$ satisfait le critère valuatif
de propreté usuel sur $K^1$, donc est propre : en effet, si $L$ est une extension de $K^1$ et $R$
un anneau de valuation de $L$ contenant $K^1$ alors $R\otimes_{K^1}K$ est un anneau gradué de valuation 
de $L\otimes_{K^1}K$ contenant $K$. Par conséquent
tout $(L\otimes_K K^1)$-point de $X$, c'est-à-dire tout
$L$-point de $X^1$,  s'étend en un unique $(R\otimes_{K^1}K)$-point de $X$, c'est-à-dire en un unique 
$R$-point de $X^1$, ce qui conclut. 

Soit $E$ un fibré vectoriel sur $X$. 
Nous dirons que $E$ est strict s'il existe un recouvrement ouvert $(V_i)$ de
$X$ tel que pour tout $i$, le fibré restreint $E|_{V_i}$ soit isomorphe à
$\mathscr O_{V_i}^{n_i}$ pour un certain $n_i$ (sans décalage de la graduation).
Il suffit pour que ce soit le cas que $\mathscr E|_{V_i}$ soit
isomorphe à $\mathscr O_{i}[\rho_i]$ pour un certain polyrayon 
de longueur finie $\rho_i=(\rho_{ij})_j$ où les $\rho_{ij}$ appartiennent à $K^\times$. 

Les foncteurs $F\mapsto F\otimes_{\mathscr O_{X^1}}\mathscr O_X$ et 
$E\mapsto E^1$ établissent une équivalence entre la catégorie des 
fibrés vectoriels sur $X^1$ et celle des fibrés vectoriels stricts sur $X$. 
Cette équivalence préserve le rang. Elle préserve aussi, dans le cas des fibrés
en droites, le fait d'être ample, d'être très ample, et d'être engendré par
ses sections globales : c'est une conséquence du fait que si $D$ est un fibré en 
droites strict sur $X$, pour toute section homogène
$s$ de $D$ il existe
$a\in K^\times$ tel que $as$ soit de degré
$1$ (et donc soit une section du fibré en droites $D^1$ sur $X^1$). 

\begin{exem}\label{exemple-deploiement-pr}
Soit $r=(r_1,\ldots, r_n)$ un polyrayon. Si les $r_i$ appartiennent tous à 
$\abs K^\times$, le schéma gradué $X=\A^r_K$ est strict, et $X^1=\A^n_K$. 
En effet, on choisit pour tout $i$ un élément $a_i$ de $K^\times$ tel que $\abs{a_i}=r_i$, et 
l'on remarque que
\[K[T_1/r_1,\ldots, T_n/r_n]=K[\tau_1,\ldots \tau_n]=K^ 1[\tau_1,\ldots, \tau_n]\otimes_{K^1}K\]
où $\tau_i=T_i/a_i$. 

Soit $\rho=(\rho_0,\ldots, \rho_n)$. un polyrayon. Si les $\rho_i$ appartiennent tous à 
$\abs K^\times$, le schéma gradué $Y=\P^\rho_K$ est strict, et $Y^1=\P^n_K$. 
De plus, le fibré $\mathscr O(1)$ sur $Y$ est lui aussi strict, le fibré en droites
correspondant sur $Y^1$ étant le fibré $\mathscr O(1)$ traditionnel. 
Cela découle en effet par exemple du cas de l'espace affine traité ci-dessus et 
de la description de $\P_A^\rho$ et de $\mathscr O(1)$ via les cartes affines standard. 
\end{exem}

\subsection{}
Soit $X$ un $K$-schéma gradué et soit $L$ une extension complète de $K$. Soit $E$ un fibré 
vectoriel sur $X$. 
Nous dirons que $L$ déploie $X$ (resp. déploie $X$ et $E$) si $X_L$ est strict (resp. si $X_L$ est strict et si 
l'image réciproque de $E$ sur $X_L$ est stricte). Nous dirons qu'un polyrayon $\rho$ déploie $X$ (resp. $X$ et $E$) 
si $K(T/\rho)$ déploie $X$ (resp. $X$ et $E$). 

Notons qu'il existe toujours un polyrayon $\rho=(\rho_i)_i$ déployant $X$ et
tous les fibrés vectoriels sur $X$ : il suffit que le groupe engendré
par $\abs{K^\times}$ et les $\rho_i$ soit égal à $\R_+^\times$ tout entier. 

Supposons $X$ de type fini et soit $(E_i)$ une famille finie de fibrés vectoriels sur $X$. 
Choisissons un recouvrement fini $(V_j)$ de $X$ par des ouverts affines tels que
$E_i|_{V_j}$ soit pour
tout $j$ isomorphe à $\mathscr O_{V_j}[\rho_{ij}]$ pour un certain polyrayon 
de longueur finie $\mathscr \rho_{ij}$. Soit $r=(r_a)$ un polyrayon de longueur finie tel que 
le groupe engendré par $\abs{K^\times}$ et les $r_a$ contienne toutes
les composantes des $\rho_{ij}$ ainsi que tous les degrés des éléments homogènes
non nuls des $\mathscr O_X(V_j)$. Le polyrayon $r$ déploie alors
$X$ et les $X_i$. 

\begin{lemm}\label{lemm-composantes-graduees}
Soit $X$ un $K$-schéma gradué. Soit $\rho$ un polyrayon déployant $X$. 
Posons $X_\rho=X\times_K K(T/\rho)$. 

\begin{enumerate}
\item Le schéma gradué $X$
est irréductible (resp. réduit) si et seulement si c'est le cas du
schéma~$X_\rho^1$. 

\item Si $X$ est noethérien et si $(X_i)$ désigne la famille des composantes irréductibles de 
$X$, les $X^1_{i\rho}$ sont les composantes irréductibles de $X^1_\rho$ et 
pour tout $i$ la multiplicité générique de $X$ le long de $X_i$ est égale à celle de $X^1_\rho$ le long 
de $X_{i\rho}^1$. 
\end{enumerate}
\end{lemm}

\begin{proof}
Commençons par \emph a). Un espace topologique possédant un recouvrement par des ouverts non vides
étant irréductible si et seulement si c'est le cas de chacun des ouverts en question, on se ramène 
au cas où $X=\Spec A$ pour une certaine $K$-algèbre graduée $A$. Le
morphisme naturel de $A$ dans $A_\rho:=A\otimes_K K(T/\rho)$ est injectif.
En conséquence, si $A_\rho$ est réduit (resp. intègre) il en va de même de
$A$. 
Par ailleurs, $A_\rho$ est un localisé de $A\otimes_K K[T/\rho]=A[T/\rho]$ par une 
partie multiplicative qui ne contient pas $0$ si $A$ est non nulle, puisque $K[T/\rho]$ s'injecte
alors dans $A[T/\rho]$. Il s'ensuit que si
$A$ est réduit (resp. intègre), il en va de même de $A_\rho$. 
On conclut alors la preuve de \emph a) en remarquant que comme $\rho$ déploie $A$, il résulte 
du lemme \ref{lemm-tenseur-K1K}
que $A_\rho^1$ est réduite (resp. intègre) si et seulement s'il en va de même de 
$A_\rho$. 

Montrons maintenant b). Il résulte de a) que les $X_{i\rho}$ sont des fermés irréductibles de 
$X_\rho$, dont la réunion est évidemment égale à $X_\rho$. De plus, comme
le morphisme naturel 
de $X_\rho$ vers $X$ est surjectif, les $X_{i\rho}$ sont deux à deux non comparables pour l'inclusion, 
et ce sont donc les composantes irréductibles de $X_\rho$. Or puisque $\rho$ déploie $X$, il déploie aussi chacun des $X_i$, 
si bien que topologiquement $X^1_\rho=X_\rho$ et $X^1_{i\rho}=X_{i\rho}$ pour tout $i$ ; les $X^1_{i\rho}$
sont donc bien les composantes irréductibles de $X^1_\rho$. 

Montrons enfin l'assertion relative aux multiplicités génériques. Celle-ci est locale au voisinage du point générique de chacune
des $X_i$. Cela permet de supposer que $X$ est irréductible et affine ; posons $A=\mathscr O_X(X)$ et notons $N$ le nilradical de $A$. 
 Soit $\eta$ son point générique
et soit $m$ la multiplicité générique de $X$, c'est-à-dire la longueur de $\mathscr O_{X,\eta}$. Quitte à retreindre encore $X$, on peut supposer
qu'il existe une suite décroissante d'idéaux homogènes $A=I_0\supset I_1\supset \ldots\supset I_m=\{0\}$ telle que 
$I_p/I_{p+1}$ soit isomorphe comme $A$-module gradué à $A/N$. La suite $(p)$ induit par extension du corps de base une suite
décroissantes d'idéaux homogènes  $A_\rho=I_{0\rho}\supset I_{1\rho}\supset \ldots\supset I_{m\rho}=\{0\}$ dont les quotients successifs
sont isomorphes à $A_\rho/N_\rho=(A/N)_\rho$. Or $(A/N)_\rho$ est réduit d'après la preuve de \emph a), si bien que $N_\rho$
est le nilradical de $A_\rho$. En conséquence, $m$ coïncide avec la longueur de $\mathscr O_{X_\rho,\xi}$ où $\xi$ est le point générique de
$X_\rho$ (qui est irréductible d'après \emph a), et donc avec la longueur de  $\mathscr O_{X^1_\rho,\xi}$ au vu du lemme
\ref{lemm-tenseur-K1K} (d), c'est-à-dire à la multiplicité générique de $X_\rho^1$.

\end{proof}

\begin{lemm}\label{lemm-deploiement-transitif}
Soit $X$ un $K$-schéma gradué. Soit $L$ une extension graduée de $K$
déployant $X$. Soit $M$ une extension graduée de $L$.  

\begin{enumerate}
\item L'extension $M$ de $K$ déploie
alors $X$ (resp. $X$ et $E$) et l'on a $X_M^1=(X_L^1)_{M^1}$.

\item Soit
$E$ un fibré vectoriel sur $X$
déployé par $L$. Le fibré $E$ est alors déployé par $M$ et $E_M^1$ est l'image réciproque de $E_L^1$ par le morphisme
de schémas $X_M^1\to X_L^1$. 
\end{enumerate}
\end{lemm}

\begin{proof}
Les deux assertions sont locales sur $X$, ce qui permet de supposer que $X$ est affine. 
Posons $A=\mathscr O_X(X)$ et, pour \emph b), $F=E(X)$. 
On a alors 
\begin{align*}
A_M=(A_L)\otimes_L M
=A_L^1\otimes_{L^1}L\otimes_L M
=A_L^1\otimes_{L^1}M
=A_L^1\otimes_{L^1}M_1\otimes_{M^1}M,
\end{align*}
où la seconde égalité provient du fait que $A_L$ est strict. 
Ceci achève la preuve de \emph a). 
Sous les hypothèses de \emph b) on a de même  
\begin{align*}
F_M=(F_L)\otimes_L M
=F_L^1\otimes_{L^1}L\otimes_L M
=F_L^1\otimes_{L^1}M
=F_L^1\otimes_{L^1}M_1\otimes_{M^1}M,
\end{align*}
ce qui montre \emph b).

\end{proof}

\subsection{}\label{ss.quasiproj-difference}
Soit $X$ un $K$-schéma gradué
quasi-projectif, ce qui veut dire qu'il existe une immersion
$i\colon X\hookrightarrow \P^r_K$ pour un certain polyrayon 
$r=(r_0,\ldots, r_n)$. Il existe alors un fibré
très ample $D$ sur $X$, à savoir 
$i^*\mathscr O(1)$. 
Soit $E$ un fibré en droites quelconque sur $X$. Soit $L$ une extension graduée de $K$ déployant 
$X$ et $E$ et telle que $r_i\in \abs{L^\times}$ pour tout $i$. Cette dernière condition assure que
$\P^r_L\simeq \P^{(1,\ldots, 1)}_L$, puis que $D_L$ est strict, le fibré 
$D_L^1$ étant l'image réciproque de $\mathscr O(1)$ par l'immersion 
$X^1\hookrightarrow \P^n_{K^1}$ induite par $i$ ; notons que $D_L^1$
est très ample. 

Appliquons ceci lorsque $L=K(T/r)$, en notant par un $r$ en indice le changemnbt de base correspondant. 
Le fibré $D_r^1$ est ample. Il existe donc $m>0$ tel que $(D_r^1)^{\otimes m}$ 
et $(D^r_1)^{\otimes m})\otimes (E_r^1)^\vee$ soient tous deux amples, et en particulier
engendrés par leurs sections globales. 
Les fibrés $D_r^{\otimes m}$ et $D_r^{\otimes m}
\otimes E_r^{\vee}$ sont alors engendrés par leurs sections globales, et
il en va dès lors de même de $D^{\otimes m}$ et
$D^{\otimes m}\otimes E^\vee$. On a ainsi écrit $E$ comme différence de deux
fibrés en droites engendrés par leurs sections globales. 

\section{Descente de la dimension et des degrés d'intersection}
Il existe une bonne théorie de la dimension pour les schémas gradués de type fini sur un corps gradué. Elle était brièvement
présentée sans démonstration dans \cite{ducros2021} ; nous allons donner un peu plus de détails ici. 

\begin{lemm}\label{lemm-dimension-graduee}
Soit $X$ un $K$-schéma gradué de type fini et soit $L$ une extension graduée de $K$ déployant $X$. 
La dimension du schéma $X_L^1$ ne dépend pas de $L$, et est 
égale à $\sup_{x\in X}\mathrm{deg.\,tr.\,}(\kappa(x)/K)$. 
\end{lemm}

\begin{proof}

Il suffit de montrer la formule
\[\dim X^1_L= \sup_{x\in X}\mathrm{deg.\,tr.\,}(\kappa(x)/K),\] puisque le terme de droite
de l'égalité est indépendant de $L$. Soit $\pi$ le morphisme de changement de base de $X_L$ vers $X$. Soit $x\in X$ et doit $d$ le degré
de transcendance de $\kappa(x)$ sur $K$. Soit $(\tau_1,\ldots, \tau_d)$
une base de transcendance de 
$\kappa(x)$ sur $K$ ; posons $r_i=\abs{\tau_i}$. La fibre 
$\pi^{-1}(x)$ s'identifie à $\Spec \kappa(x)\otimes_K L$, qui est non nul. 
Elle est donc non vide. 
On peut par ailleurs écrire 
\[\kappa(x)\otimes_K L=\kappa(x)\otimes_{K(\tau)}(K(\tau)\otimes_K L).\] 
Posons $A=K(\tau)\otimes_K L$. La fibre $\pi^{-1}(x)$ est muni par ce qui précède d'un morphisme
fini surjectif sur $\Spec A$. Soit $y\in \Spec A$. Son corps résiduel gradué est engendré en 
tant qu'extension de $L$ par les images des $\tau_i$, donc est de degré de  transcendance au plus
$d$ sur $L$. De plus, $A$ est intègre (c'est un localisé non nul de l'anneau gradué de polynômes
$K[\tau]\otimes_K L\simeq L[T_i/r_i]$) et 
le corps résiduel gradué du point générique de $\Spec A$ est le corps
gradué des fractions de $A$, qui s'identifie à $L(T_i/r_i)$ et est donc
de degré
de transcendance $d$ sur $L$. Par conséquent
$\sup_{y\in \Spec A}\mathrm{deg.\,tr.\,}(\kappa(y)/L)=d$, ce qui entraîne que
$\sup_{z\in \pi^{-1}(x)}\mathrm{deg.\,tr.\,}(\kappa(z)/L)=d$. 
Ceci valant pour tout $x\in X$ il vient 
\[\sup_{x\in X}\mathrm{deg.\,tr.\,}(\kappa(x)/K)=\sup_{z\in X_L}
\mathrm{deg.\,tr.\,}(\kappa(z)/L).\]
Il reste à montrer que ce dernier terme est égal à 
$\dim X_L^1$. Soit $z\in X_L^1$ (rappelons que topologiquement $X_L^1=X_L$). 
Le corps gradué résiduel $\kappa(z)$ est alors égal 
à $\kappa(z)^1\otimes_{L^1}L$, et son degré de transcendance sur $L$ est dès lors
égal
à $\mathrm{deg.\,tr.\,}(\kappa(z)^1/L^1)$. Mais $\kappa(z)^1$ est le corps résiduel classique
du point $z$ du $L^1$-schéma de type fini $X_L^1$, si bien que 
\[ \sup_{z\in X_L}
\mathrm{deg.\,tr.\,}(\kappa(z)/L)=\sup_{z\in X^1_L}
\mathrm{deg.\,tr.\,}(\kappa(z)^1/L^1)=\dim X_L^1,\]
ce qui achève la démonstration. 
\end{proof}

\begin{defi}
Soit $K$ un corps gradué et soit $X$ un $K$-schéma gradué de type fini. La dimension de $X$ (au-dessus de $K$) 
est le supremum des $\mathrm{deg.\,tr.\,}(\kappa(x)/K)$ pour $x$ parcourant $X$.
\end{defi}

Il résulte du lemme \ref{lemm-dimension-graduee} que la dimension de $X$ est finie dès que $X$ est non vide (elle vaut 
$-\infty$ si $X$ est vide). Elle est invariante par changement de corps de base : cela résulte du fait que si $L$ est une extension 
graduée quelconque de $K$, toute extension graduée $M$ de $L$ déployant $X_L$ est une extension graduée de $K$ déployant $X$.

Soit $Y$ une composante irréductible de $X$, soit $\eta$ son point générique et soit $\zeta$ un point 
de $Y$ distinct de $\eta$. Soit $Z$ l'adhérence de $\zeta$ dans $Y$. Soit $\rho$ un polyrayon 
déployant $Y$. Le lemme \ref {lemm-composantes-graduees} assure que $Y_\rho$ est irréductible, et 
$Z_\rho$ en est un fermé strict par surjectivité
de $Y_\rho\to Y$. En conséquence $Z_\rho^1=Z_\rho$ est un fermé
strict du schéma irréductible $Y_\rho^1=Y_\rho$ (les égalités écrites étant purement topologiques)
si bien 
que $\dim Z_\rho^1<\dim Y_\rho^1$, et donc que $\dim Z<\dim Y$. Ceci valant pour tout $\zeta$ on voit que
$\sup_{y\in Y}\mathrm{deg.\,tr.\,}(\kappa(y)/K)$ n'est atteint qu'en $\eta$, et $\dim Y$ est en conséquence égal
au degré de transcendance de $k(\eta)$ sur $K$. Il s'ensuit
que $\dim X$ est le supremum des $\mathrm{deg.\,tr.\,}(\kappa(x)/K)$ où $x$ parcourt l'ensemble des points
génériques des composantes irréductibles de $X$.

\begin{exem}
Soit $r$ un réel strictement positif qui n'est pas de torsion
modulo $K^\times$. Posons $L=K(T/r)$. C'est une extension de $K$, mais aussi une 
$K$-algèbre de type fini (isomorphe à $K[T/r, rS]/(ST-1)$). 
Posons $X=\Spec L$. 
La dimension de Krull de $X$ est nulle et sa dimension comme schéma gradué
ce type fini sur $L$ est dès lors nulle puisque c'est un $L$-schéma strict, mais sa dimension 
comme $K$-schéma gradué
de type fini vaut $1$. On peut le déduire du fait que $L$
est de degré de transcendance $1$ sur $K$, ou le voir
comme suit. On choisit une extension graduée $F$ de $K$ possédant un élément
non nul $a$ homogène de degré $r$. Le schéma gradué $X_F$ est alors égal à 
$\Spec F[T/r, rS]/(TS-1)$ et est donc isomorphe à $\Spec F[U,V]/(UV-1)$ (poser
$U=T/a$ et $V=aS$). Par conséquent $X_F$ est strict et $X_F^1=\gm[F^1]$, si bien 
que $\dim X=1$. 
\end{exem}

\begin{lemm}\label{lemme-suite-reguliere}
Soit $X$ un $K$-schéma gradué intègre de type fini sur un corps 
gradué $K$ et soit $n$ sa dimension. 
Soit $Y$ un fermé de Zariski irréductible de $X$
et soit $m$ sa dimension. Il existe un ouvert $U$ de $X$
rencontrant $Y$ et $n-m$ éléments homogènes 
$f_1,\ldots, f_{n-m}$ de $\mathscr O_X(U)$ tels que 
$U\cap Y$ soit égal au lieu des zéros des $f_i$.
\end{lemm}

\begin{proof}
Posons $d=n-m$ et raisonnons par récurrence sur $d$. On suppose donc le résultat
vrai en codimension $<d$. 
Si $d=0$ la famille vide des $f_i$ convient. Supposons $d>0$. Le fermé $Y$ est alors strict. 
Soit $U$ un ouvert affine de $X$ rencontrant $Y$. L'intersection $U\cap Y$ est un fermé strict de 
$U$, si bien qu'il existe un élément homogène $f\in \mathscr O_X(U)$ non nul et s'annulant sur $Y$. 
Soit $Z$ une composante irréductible du lieu des zéros de $f$ contenant $U\cap Y$ ; 
quitte à restreindre $U$, on peut supposer que $Z$ est exactement
le lieu des zéros de $f$. La dimension de $Z$ est égale à
$n-1$ : le lemme \ref{lemm-composantes-graduees} permet en effet de se ramener au cas non gradué, et c'est
alors une conséquence du Hauptidealsatz. L'hypothèse de récurrence assure l'existence d'un ouvert $V$
de $Z$ rencontrant $Y$ et d'éléments homogènes  $g_2,\ldots, g_d$ de 
$\mathscr O_Z(V)$ telles que $V\cap Y$ soit
le lieu des zéros des $g_i$. Quitte à restreindre $U$ et $V$
on peut supposer que $V=U\cap Y$. Chacune des $g_i$ se relève alors en un élément $f_i$ de 
$\mathscr O_X(U)$, et la famille $(f,f_2,\ldots, f_d)$ répond aux exigences de l'énoncé. 
\end{proof}

\subsection{}
Soit $K$ un corps (classique) et soit $X$ une $K$-variété algébrique propre, 
purement de dimension $n$. Si $D$ est un fibré en droites sur $X$, 
la théorie de l'intersection fournit un opérateur $c_1(D)\cap (\cdot)$ de degré $(-1)$
sur $\mathrm{CH}^\bullet(X)$ ; lorsque $D$ varie, les opérateurs ainsi obtenus commutent deux à deux. 
Si $D_1,\ldots, D_n$ sont $n$ fibrés en droites
sur $X$, on obtient en particulier une classe
 $\alpha=(c_1(D_1)\cap \dots\cap  c_1(D_n) \cap [X]$
dans $\mathrm{CH}^n(X)=\mathrm{CH}_0(X)$ ; si $p$ désigne le morphisme propre structural de $X$ vers $\Spec K$, la classe
$p_*\alpha$ est un élément de $\mathrm{CH}_0(\Spec K)$, que l'on identifie à $\Z$ en prenant
$[\Spec K]$ comme 
générateur. L'entier relatif $p_*\alpha$ est appelé le degré de la classe $\alpha$, ou encore le degré d'intersection 
des $L_i$, et nous le noterons $\int_X c_1(D_1)\cdots c_1(D_n)$.

Nous nous proposons maintenant
de définir les nombres d'intersection dans le cas gradué, sans y construire toute la théorie de l'intersection (ce qui serait selon toute vraisemblance
possible mais très fastidieux); nous allons procéder par descente à partir du cas non gradué. 

\begin{prop}\label{prop-intersection-graduee}
Soit $K$ un corps gradué, soit $X$ un $K$-schéma gradué propre et de type fini purement de dimension $n$, 
et soit $D_1,\ldots, D_n$ des fibrés en droites gradués sur $X$. Soit
$L$ une extension graduée de $K$ déployant $X$ et les $D_i$. 
Le nombre d'intersection $\int_{X^1_L}c_1(D^1_{1,L})\cdots c_1(D^1_{n,L})$ ne dépend alors pas de $L$. 
\end{prop}

\begin{proof}
Supposons donnée une seconde extension
graduée $F$
de $K$ déployant $X$ et les $D_i$ et montrons que l'entier obtenu en utilisant $F$ est le même 
que celui obtenu en utilisant $L$. Il suffit pour ce faire de
montrer que les deux coïncident avec l'entier
obtenu en utilisant une extension composée $M$ de $F$ et $L$ (qu'on fabrique en considérant
le quotient de l'anneau gradué non nul $F\otimes_K L$ par un idéal homogène maximal).
Par symétrie, il suffit de montrer que l'entier obtenu en utilisant $L$ est le même que celui obtenu en utilisant 
$M$. Mais cela découle alors
du lemme \ref{lemm-deploiement-transitif}
et de l'invariance du degré d'intersection schématique par changement du corps de base. 
\end{proof}

\begin{defi}
Soit $K, X, (D_i),$ et $L$ comme dans la proposition ci-dessus. L'entier 
$\int_{X_L^1}c_1(D^1_{1,L})\cdots c_1(D^1_{n,L})$ 
(qui ne dépend pas du choix de~$L$) sera noté
$\int_X C_1(D_1)\cdots c_1(D_n)$. 
\end{defi}

\begin{prop}\label{proprietes-intersection-graduee}
Soit $X$ un $K$-schéma gradué propre et purement de dimension $n$. 

\begin{enumerate}
\item 
L'application \[\mathrm{Pic}(X)^n\to \Z, \quad (D_1,\ldots, D_n)\mapsto
\int_X c_1(D_1)\cdots c_1(D_n)\] est symétrique et multi-linéaire. 

\item Soit $L$ une extension graduée de $K$.
On a \[
\int_{X_L} c_1(D_{1,L})\cdots c_1(D_{n,L})
=\int_X c_1(D_1)\cdots c_1(D_n).\]

\item Soit $X_1,\ldots, X_m$
les composantes irréductibles réduites de $X$ et pour tout $j$, soit $m_j$
la multiplicité générique de $X$ le long de~$X_j$.
Soit $D_1,\ldots, D_n$ des fibrés en droites sur~$X$. 
On a 
\[\int_Xc_1(D_1)\cdots c_1(D_n)=\sum_j m_j\int_{X_j} c_1(D_1|_{X_j})\cdots c_1(D_{n}|_{X_j}).\]

\item Supposons $X$ intègre et soit $Y\to X$ un morphisme propre et génériquement fini, de degré
générique $d$. Pour tout $i$, soit $E_i$ l'image réciproque de $D_i$ sur $Y$. 
On a alors 
\[\int_Y c_1(E_1)\cdots c_1(E_n)=d \int_X c_1(D_1)\cdots c_1(D_n).\]
\end{enumerate}
\end{prop}

\begin{proof}
Soit $j$ un indice entre $1$ et $n$, soit $D_1,\ldots, D_n$ et $E_j$ des
fibrés en droites sur $X$,  et soit $\sigma$ une permutation de $\{1,\ldots, n\}$. 
Soit $L$ une extension graduée de $K$ déployant $X$, les $D_i$ et $E_j$. 
On a 
\begin{align*}
\int_X c_1(D_{\sigma(1)})\cdots c_1(D_{\sigma(n)})&=
\int_{X^1_L} c_1(D^1_{\sigma(1),L})\cdots c_1(D^1_{\sigma(n),L})\\
&=\int_{X^1_L} c_1(D^1_{1,L} \cdots c_1(D^1_{n,L})\\
&=\int_X c_1(D_1)\cdots c_1(D_n)
\end{align*}

et
\[
\int_X c_1(D_1)\cdots c_1(D_j\otimes E_j)\cdots c_1(D_n)\]

\begin{align*}
&=
\int_{X^1_L} c_1(D^{1,L})\cdots c_1(D^1_{j,L}\otimes E^1_{j,L})\cdots c_1(D^1_{n,L})\\
&=\int_{X^1_L} c_1(D^{1,L})\cdots c_1(D^1_{j,L})\cdots c_1(D^1_{n,L})
+\int_{X^1_L} c_1(D^{1,L})\cdots c_1(E^1_{j,L})\cdots c_1(D^1_{n,L})\\
&=\int_X c_1(D_1)\cdots c_1(D_j)\cdots c_1(D_n)
+\int_X c_1(D_1)\cdots c_1(E_j)\cdots c_1(D_n),\end{align*}
ce qui achève de montrer \emph a). 
L'assertion  \emph b) provient du fait que si $M$ est une extension graduée
de $L$ déployant $X$ et les $D_i$, les membres de l'identité à établir sont tous deux égaux à
\[\int_{X_M^1}c_1(D_{1,M}^1)\cdots c_1(D^1_{n,M}).\]
Pour montrer \emph c), choisissons un polyrayon $r$ déployant $X$ et les $D_i$, et notons
par un $r$ en indice l'extension des scalaires à $k(T/r)$.  
On a alors
\begin{align*}
\int_Xc_1(D_1)\cdots c_1(D_n)&=
\int_{X_r^1}c_1(D^1_{1,r})\cdots c_1(D^1_{n,r})\\
&=\sum_j m_j \int_{X_{j,r}^1}c_1(D^1_{1,r}|_{X_{j,r}^1})\cdots c_1(D^1_{n,r}|_{X_{j,r}^1})\\
&=\sum_j m_j\int_{X_j} c_1(D_1|_{X_j})\cdots c_1(D_{n}|_{X_j}),
\end{align*}
où la troisième égalité provient de la formule à montrer dans le cas
classique (où elle est connue) et du lemme \ref{lemm-composantes-graduees} ; 
ceci achève de prouver \emph c). 

Pour montrer \emph d), choisissons un polyrayon $r$ déployant $X, Y$ et les $D_i$. 
Le schéma gradué $X_r^1$ est intègre et $Y_r^1\to X_r^1$ est propre, et il est génériquement fini de degré
générique $d$ d'après le lemme \ref{lemm-tenseur-K1K} \emph f). 
ll vient 
\begin{align*}
\int_Y c_1(E_1)\cdots c_1(E_n)&=
\int_{Y_1^r} c_1(E^1_{1,r})\cdots c_1(E^1_{n,r})\\
&=d \int_{X_1^r} c_1(D^1_{1,r})\cdots c_1(D^1_{n,r})\\
&=d \int_X c_1(D_1)\cdots c_1(D_n),
\end{align*}
ce qu'il fallait démontrer. 

\end{proof}

\begin{rema}\label{rema-intersection-decalage}
Soit $X$ un $K$-schéma gradué purement de dimension $n$, soit
$(D_i)$ une famille de $n$ fibrés en droites gradués sur $X$, et soit $(r_i)$ un polyrayon. 
Alors 
$\int_Xc_1(D_1)\cdots c_1(D_n)=\int_X c_1(D_1[r_1])\cdots c_1(D_n[r_n]))$. 
En effet, 
il suffit pour le voir de considérer une extension graduée $L$ de $K$ telle que $r_i\in \abs{L^\times}$ pour
tout $i$. Sur le schéma gradué $X_L$ on a alors $D_i[r_i]\simeq D_i$ pour tout $i$ (via une homothétie
de rapport $\lambda_i\in K^\times$ avec $\abs{\lambda_i}=r_i$). 
Compte-tenu de l'invariance du nombre d'intersection par extension des scalaires, on en déduit que 
$\int_X c_1(D_1[r_1])\cdots c_1(D_n[r_n])=\int_X c_1(D_1)\cdots c_1(D_n)$. 
\end{rema}

  \begin{exem}\label{exemple-c1-o1}
Soit $r=(r_0,\ldots, r_n)$ un polyrayon et soit $L$ une extension graduée de $K$ telle que les
$r_i$ appartiennent à $K^\times$. 
Il résulte de l'exemple \ref{exemple-deploiement-pr}
que 
$\int_{\P^r_K}c_1(\mathscr O(1))^n=\int_{\P_{L_1}^n}c_1(\mathscr O(1))^n$, si bien que 
$\int_{\P^r_K}c_1(\mathscr O(1))^n=1$. 
\end{exem}

\subsection{}\label{def-numeff}
Soit $(D_1,\ldots, D_d)$ une famille de fibrés en droites sur un $K$-schéma gradué propre $X$. 
On dit que $c_1(D_1)\ldots c_1(D_d)$ est numériquement effectif si pour tout sous-schéma gradué fermé
intègre $Y$ de $X$ purement de dimension $d$ le nombre $\int_Y c_1(D_1|_Y)\ldots c_1(D_d|_Y)$ est positif ou nul. 
On dit qu'un fibré en droites $D$ est numériquement effectif si $c_1(D)$ est numériquement effectif. 

Supposons que  $c_1(D_1)\ldots c_1(D_d)$ soit numériquement effectif. Soit $L$ une extension graduée de $K$, soit $Z$
un $L$-schéma gradué propre et soit $f\colon Y\to X$ un $K$morphisme. Le produit
$c_1(f^*D_1)\ldots c_1(f^*D_d)$ est alors numériquement effectif. En effet si $L=K$ cela se démontre exactement comme dans le
cas classique ; il reste donc à traiter le cas où $Z=X_L$, et par
passage à la limite on peut supposer $L$ de type fini. 
Si $L$ est une extension finie de $K$ la démonstration classique s'adapte encore. Il reste finalement
par récurrence sur le degré de transcendance à traiter le cas où $L=K(T/r)$ pour un certain $r>0$. Soit $Y$
un sous-schéma gradué fermé intègre purement de dimension $d$ de $X_L$ et soit $p$ la projection de $X\times_K \A^r_K$ vers $\A^r_K$. 
Le $L$-schéma $Y$ s'identifie à la fibre générique de $p$. Soit $Y'$ son adhérence réduite dans $X\times_K \A^r_K$. C'est un $K$-schéma
propre sur $\A^r_K$. Pour tout $t\in \A^r_K$, la fibre $Y'_t$ est propre et purement de dimension $d$, 
et $\int_{Y'_t}c_1(D_1|_{Y'_t})\ldots c_1(D_d|_{Y'_t})$ ne dépend pas de $t$ : on le vérifie en se ramenant au cas classique par extension des scalaires
à $K(S/\rho)$ pour un certain polyrayon $\rho$ déployant la situation, dans lequel l'assertion découle simplement de la
«conservation  du nombre» (\cite{fulton98}, théorème 10.2). 
En particulier $\int_Y c_1(D_1|Y)\ldots c_1(D_d|_Y)=\int_{Y'_0} c_1(D_1|_{Y'_0})\ldots c_1(D_d|_{'Y_0})$ et ce dernier 
est positif ou nul puisque $c_1(D_1)\ldots c_1(D_d)$ est numériquement effectif, ce qui conclut.

\section{Espaces de Zariski-Riemann sur les corps gradués}

\subsection{}\label{def-affine-zr}
Soit $L$ un corps gradué. Notons $\P_L$ l'ensemble des classes d'équivalences de valuations sur $L$. 
En associant à tout élément $v$ de $\P_L$ la restriction à $L^\times$ de l'indicatrice de l'anneau 
$\mathscr O_v$ on identifie $\P_L$ à un sous-ensemble $\{0,1\}^{L^\times}$. Les axiomes définissant une valuation 
sont des conditions fermées, qui font donc de $\P_L$ un espace compact et totalement discontinu lorsqu'on le munit
de la topologie induite par la topologie produit (qui est dite \textit{constructible}). 

En pratique on utilisera une topologie plus grossière que la topologie constructible, davantage adaptée
aux applications à la géométrie algébrique et à la géométrie analytique non archimédienne. Pour ce faire, introduisons
la notation suivante. Si $S$ est une partie de $L^\times$, on notera $\P_L\{S\}$ l'ensemble des valuations $v$ 
telles que
$S\subset \mathscr O_v$, c'est-à-dire encore telles que $v(\lambda)\leq 1$ pour tout $\lambda \in S$ ; 
remarquons que $\P_L\{S\}\cap  \P_L\{T\}=
\P_L\{S\cup T\}$ et que $\P_L\{\emptyset\}=\P_L$. On munira alors $\P_L$ de la topologie engendrée par les 
parties de la forme $\P_L\{S\}$ où $S$ est un sous-ensemble \textit{fini}
de $L^\times$, qui seront appelés \textit{ouverts affines} de $\P_L$. Remarquons que la topologie 
constructible de $\P_L$ est engendrée par les ouverts affines et leurs complémentaires. Les éléments
de l'algèbre de Boole engendrée par les ouverts affines de $\P_L$ sont précisément les parties à la fois ouvertes
et fermées de $\P_L$ pour la topologie constructible, qu'on qualifiera plus brièvement de parties constructibles.

Soit $K$ un corps gradué
muni d'un plongement dans $L$. 
Soit $v$ une valuation sur $K$. Nous noterons $\zr L{(K,v)}$ le sous-ensemble de $\P_L$ formé des
valuations  dont la restriction à $K$ est égale à $v$, et nous le munirons de la topologie
induite par celle de $\P_L$. Si $S$
est une partie de $L^\times$ nous noterons
$\zr L{(K,v)}\{S\}$ l'intersection de $\zr L{(K,v)}$ et $\P_L\{S\}$ ; la topologie de 
 $\zr L{(K,v)}$ est donc engendrée par les parties de la forme $\zr L{(K,v)}\{S\}$
 où $S$ est un sous-ensemble \textit{fini}
de $L^\times$ ; les ouverts de $\zr L{(K,v)}$ qui sont précisément de cette forme 
seront qualifiés d'affines, et ce sont les traces d'ouverts affines de $\P_L$. 
La topologie de $\zr L{(K,v)}$ induite par la topologie constructible de $\P_L$
sera appelée la topologie constructible de $\zr L{(K,v)}$. les éléments
de l'algèbre de Boole engendrée par les ouverts affines de $\zr L{(K,v)}$ sont précisément les parties à la fois ouvertes
et fermées
de $\zr L{(K,v)}$ pour la topologie constructible, qu'on qualifiera plus brièvement de parties constructibles.

Remarquons que $\zr L{(K,v)}$ est, en tant sous-ensemble de $\P_L$, défini par une conjonction (en général infinie)
de conditions fermées pour la topologie constructible. Il est donc compact pour 
la topologie constructible. Chacun de ses ouverts affines étant un ouvert fermé pour la topologie constructible,
il est compact pour cette dernière, et donc quasi-compact pour la topologie plus grossière
dont on a muni $\zr L{(K,v)}$. 

Soit $(K',v')$ une extension de corps gradués valués de $(K,v)$. Soit $L'$ une extension 
de $K$  munie de deux $K$-plongements
$K'\hookrightarrow L'$ et $L\hookrightarrow L'$. 
La restriction des valuations induit une application de
$\zr{L'}{(K',v')}$ vers $\zr L{(K,v)}$, qui est surjective lorsque $K'=K$. 
Si $S$ est une partie finie de $L^\times$, l'image réciproque de $\zr L{(K,v)}\{S\}$ sur $\zr L{K'}{(K',v')}$ est égale à $\zr{L'}{(K',v')}\{S\}$. 
L'application de $\zr{L'}{(K',v')}$ vers $\zr L{(K,v)}$ est donc continue et quasi-compacte. 

Lorsque $v$ est la valuation triviale, on écrira simplement $\zr LK$ au lieu de $\zr L{(K,v)}$. C'est essentiellement avec ce type d'espaces
de Zariski-Riemann que nous allons travailler, mais nous avons introduit le cas d'un corps gradué de base arbitraire parce que nous en aurons
tout de même besoin plus bas. 

Remarquons que $\zr LK$ possède un unique point générique, à savoir la valuation triviale.

\begin{exem}
Soit $k$ un corps classique, et soit $r>0$. Notons $K$ le corps gradué égal à
$k[t,t^{-1}]$ de la graduation pour laquelle $at^i$ est homogène
de degré $r^i$ pour tout $a\in k^\times$ et tout $i\in \Z$. Il résulte des définitions que
$\P_{K/k}$ contient exactement trois valuations : la valuation triviale $v$, la valuation $v^-$ qui prend une valeur
strictement inférieure à $1$ sur $t$, et la valuation $v^+$  qui prend une valeur strictement supérieure
à $1$ sur $t$. La valuation $v$ est dense et ouverte, et les valuations $v^-$ et $v^+$ sont fermées : 
le singleton $\{v\}$ est égal à $\zr Kk\{t,t^{-1}\}$, le singleton 
$\{v^+\}$ est le complémentaire de $\zr Kk\{t\}$, et le singleton $\{v^-\}$
est le complémentaire de $\zr Kk\{t^{-1}\}$.
\end{exem}

\subsection{}
On munit l'espace topologique $\zr LK$ d'un faisceau
d'anneaux gradués $\mathscr O_{\zr LK}$ défini par la formule
\[   \mathscr O_{\zr LK}(U)=\bigcap_{v\in U}\mathscr O_ v.\] Sa fibre en un élément $v$
de $\zr LK$ est précisément l'anneau gradué local $\mathscr O_v$; 
Il est en effet clair que
$\mathscr O_{\zr LK, v}
\subset \mathscr O_v$;
pour l'inclusion réciproque, il suffit de noter que si $\lambda$ est un élément de $L^\times\cap \mathscr O_v$ alors
$\lambda$ appartient à $\mathscr O(\zr LK\{\lambda\})$ tautologiquement, 
et comme $v\in \zr LK\{\lambda\}$ puisque $\lambda\in \mathscr O_v$, on voit
que $\lambda$ appartient $\mathscr O_{\zr LK, v}$. 
Si $U$
est un ouvert de $\zr LK$, nous désignerons par $\mathscr O_U$ la restriction de $\mathscr O $
à $U$.

Soit $K'$ une extension de $K$ et 
soit $L'$ une extension de $K$ munie de deux $K$-plongements
$K'\hookrightarrow L'$ et $L\hookrightarrow L'$. 
Soit $\pi$ l'application continue de $\zr{L'}{K'}$ vers $\zr LK$. L'inclusion 
de $L$ dans $L'$ induit alors un morphisme 
de faisceaux en anneaux gradués de $\mathscr O_{\zr LK}$ vers $\pi_*\mathscr O_{\zr{L'}{K'}}$. 

\subsection{}\label{ss-zr-torsion}
Supposons que $\abs{L^\times}/\abs{K^\times}$ est de torsion ; c'est 
en particulier le cas dès que $\abs{L^\times}=\abs{K^\times}$, 
c'est-à-dire dès que le morphisme naturel de $L^1\otimes_{K^1}K$ vers $L$ est
un isomorphisme. Pour tout $\lambda\in L^\times$ et tout 
entier $N>0$ tel que $\abs \lambda ^N\in \abs{K^\times}$, 
il existe un couple $(a,b)\in (L^1)^\times\times
K^\times$ tel que $\lambda^N=ab$ ; le couple $(a,b)$ est unique à multiplication près par un élément
de la forme $(c,c^{-1})$ où $c\in (K^1)^\times$. 
Toute valuation $u$
appartenant à $\zr {L^1}{K^1}$ s'étend alors uniquement en
une valuation $v$ appartenant à $\zr LK$, définie comme suit : si pour un élément
$\lambda$ de $L^\times$ on a $\lambda^N=ab$ avec $N>0$, $a\in (L^1)^\times$ et $b\in K^\times$,
on pose $v(\lambda)=u(a)^{1/N}$ ; c'est la seule valeur possible pour $v(\lambda)$ puisqu'on doit avoir
$v(b)=1$, et l'on vérifie 
que $u(a)^{1/N}$ ne dépend 
bien que de $\lambda$
et que cette construction définit bien
une valuation sur $L$ dont l'anneau contient $K$, et dont la restriction à $L^1$ est égale 
à $u$. 
Par ailleurs, si $(\lambda_1,\ldots, \lambda_n)$ est une famille finie d'éléments de $L^\times$,
et si l'on se donne pour tout $i$
une écriture $\lambda_i^{N_i}=
a_ib_i$ avec $N_i>0, a_i\in (L^1)\times$ et $b_i\in K^\times$, on
a $\zr LK\{\lambda_1,\ldots, \lambda_n\}=\zr LK\{a_1,\ldots, a_n\}$. 
Il s'ensuit que l'application continue naturelle $\pi$ de $\zr LK$ vers $\zr {L^1}{K^1}$ est un 
homéomorphisme. La composante homogène $\mathscr O_{\zr LK}^1$ de
$\mathscr O_{\zr LK}$
s'identifie alors
à $\pi^{-1}\mathscr O_{\zr {L^1}{K^1}}$ (qui est 
un faisceau d'anneaux classiques puisque $L^1$ et $K^1$ sont trivialement gradués), 
et l'on a $\mathscr O_{\zr LK}
=\mathscr O_{\zr LK}^1\otimes_{K^1}K$.

\subsection{}\label{def-strict-zr}
On ne suppose plus que $\abs{L^\times}$ est de torsion modulo $\abs{K^\times}$.
Il résulte du paragraphe précédent 
que l'application continue naturelle de $\zr {L^1\otimes_{K^1}K}K$ vers $\zr{L^1}{K^1}$ 
est bijective. Puisque l'application continue de $\zr LK$ vers $\zr {L^1\otimes_{K^1}K}K$
est surjective, l'application continue de $\zr LK$ vers $\zr {L^1}{K^1}$ est surjective. 

Nous dirons qu'un ouvert $U$ de $\zr LK$ est \textit{strict}
s'il est égal à l'image réciproque d'un ouvert de $\zr {L^1}{K^1}$ (qui est alors
uniquement déterminé : par surjectivité, c'est nécessairement l'image de $U$). Cela revient
à demander que $U$ soit une réunion d'ouverts de la forme $\zr LK\{S\}$ où $S$
est un sous-ensemble fini de $(L^1)^\times$. Il suffit pour que ce soit le cas que $U$ soit une réunion d'ouverts de la forme
$\zr LK\{S\}$ où $S$
est un sous-ensemble fini de $L^\times$ tel que $\abs s\in \abs{K^\times}^\Q$ pour tout $\lambda \in S$. 
En effet dans ce cas il existe pour tout $s\in S$ 
 un entier $N(s)>0$ et un élément $\mu(s)$ de $K^\times$
tel que $\abs{s^{N(s)}/\mu(s)}=1$, et 
$\zr LK\{S\}=\zr LK\{s^{N(s)}/\mu(s)\}_s$,

Soit $\lambda\in L^\times$ tel que $\abs{\lambda}\notin \abs{K^\times}^\Q$. 
Dans ce cas $\lambda$ est transcendant sur $L^1\otimes_{K^1}K$ (\ref{deg-extension-graduee}), si bien que tout valuation 
appartenant à $\zr{L^1\otimes_{K^1}K}K$ admet un prolongement à $L$ prenant une valeur strictement supérieure à $1$
en $\lambda$. Il en résulte que si $U$ est un ouvert strict et non vide de $\zr LK$, il n'est pas contenu dans $\zr LK\{\lambda\}$. 

On en déduit que si $S$ est un sous-ensemble
de $L^\times$ tel que $U:=\zr LK\{S\}$ contienne un ouvert strict et non vide de 
$\zr LK$ 
alors $\abs s\in \abs{K^\times}^\Q$ pour tout $s\in S$. 
C'est en particulier le cas si $S$ est fini et si l'ouvert affine $U$ est lui-même strict.
On peut alors écrire
$U=\zr LK\{S\}=\zr LK\{s^{N(s)}/\mu(s)\}_s$
en définissant $N(s)$ et $\mu(s)$ comme au paragraphe précédent, si 
bien que $U$ est l'image réciproque d'un ouvert affine de
$\zr{L^1}{K^1}$.

\subsection{}\label{ss.vocabulaire-foncteur-polyedral}
Notre but est maintenant de généraliser le théorème 1.3 de \cite{ducros2012b} en en donnant une version 
lorsque le corps gradué de $K$ de base est muni d'une valuation $v$, puis en établissant 
un résultat d'uniformité en $v$. Nous allons commencer par un peu de vocabulaire. Si $G$ est un groupe abélien ordonné quelconque
(noté multiplicativement) on note $\mathsf D_G$ 
la catégorie des groupes abéliens divisibles ordonnés non triviaux $H$ munis  d'un morphisme injectif croissant de $G$ vers $H$, 
les flèches étant les $G$-morphismes injectifs croissants. On note $\Gamma_G$ le foncteur d'oubli de 
$\mathsf D_G$ vers $\mathsf{Ens}$.
Soit $n$ un entier. Un sous-foncteur $D$ de $\Gamma_G^n$ est dit \textit{définissable} s'il est définissable dans le langage des 
groupes ordonnés à paramètres dans $G$, c'est-à-dire s'il peut être décrit par conjonction et disjonction finies 
d'inégalités de la forme $\phi\leq 1$ ou $\phi<1$, où $\phi$ est une transformation naturelle de la forme $(x_1,\ldots, x_n)\mapsto 
a\prod x_i^{r_i}$ où $a\in G$ et où les $r_i$ sont rationnels. Un sous-foncteur 
$D$ de $\Gamma_G^n$ 
de $\Gamma_G^n$ est dit \textit{polyédral} s'il peut être défini par conjonction et disjonction finies d'inégalités de la forme $\phi\leq 1$
avec $\phi$ comme ci-dessus. Soit $H$ un objet de $\mathsf D_G$. Par modèle-complétude de {\sc doag}, $D\mapsto
D(H)$ établit une bijection entre l'ensemble des sous-foncteurs définissables (resp. polyédraux) de 
$\Gamma_G^n$ et l'ensemble des parties de $H^n$ que l'on peut décrire par conjonction et disjonction finies 
d'inégalités de la forme $\phi\leq 1$ ou $\phi<1$ (resp. $\phi\leq 1$) où $\phi$
est comme ci-dessus. 

Soit $d$ un entier entre $0$ et $n$. Un $d$-pavé
de $\Gamma_G^n$ est un sous-foncteur de $\Gamma_G^n$ de la forme 
\[H\mapsto M(\{(h_1,\ldots, h_n), a_i< h_i< b_i\;\text{pour}\;1\leq i\leq d\;\text{et}\;h_i=c_i\;\text{pour}
\;d< i\leq n\})\] où
les $a_i$, les $b_i$ et les $c_i$ appartiennent à $G$, où $ai<b_i$ pour tout $i$ entre $1$ et $d$ et où $M$ est une matrice de
$\GL_n(\Q)$. C'est un sous-foncteur définissable de $\Gamma_G^n$. 
La dimension d'un sous-foncteur
définissable non vide $D$ de $\Gamma_G^n$ est le plus grand entier $d$ tel que $D$ contienne un
$d$-pavé, et elle peut être calculée sur un objet $H$ de $\mathsf D_G$. 
Elle est  égale à 
\[\max_{H\in \mathsf D_G, (h_1,\ldots, h_n)\in \mathscr D(H)}
\dim_{\Q}\langle G, h_1,\ldots, h_n\rangle^\Q/G^\Q.\]

Si $H\in \mathsf D_G$ et si $h\in D(H)$, la dimension de $D$ en  $h$ est le minimum
des dimensions des $P\cap D$ où $P$ est un $n$-pavé tel que $h\in P(H)$. 
Lorsque $D$ est non vide, sa dimension est 
égale à $\sup_{H\in \mathsf D_\Gamma, h\in D(H)}\dim_h D$, et l'on étend cette définition au cas vide en
posant $\dim \emptyset=-\infty$. On dit que $D$ est purement de dimension $d$ si
$\dim_h D=d$ pour tout $H\in \mathsf D_G$ et tout $h\in D(H)$ ; cela peut se vérifier
sur un objet $H$ donné de $\mathsf D_G$.

\begin{theo}\label{theo.foncteurs-polyedraux}
Soit $K$ un corps gradué valué, 
soit $L$ une extension
graduée de~$K$ et soit $\mathscr V$ un ouvert quasi-compact
de $\P_L$.
Soit $(f_1,\dots, f_n)$ une famille d'éléments de $L^\times$
et soit~$d$ le degré de transcendance gradué
de $K(f_1,\dots, f_n)$ sur~$K$.

\begin{enumerate}
\item
Soit $v$ une valuation sur $K$. 
Soit $D_{(K,v),L,\mathscr V ,(f_i)}$
le sous-foncteur de~$\mathbf G_{v(K^\times)}^n$ qui associe à tout $H \in \mathsf D_{v(K^\times)}$ 
l'ensemble des $n$\nobreakdash-uplets $(h_1,\dots, h_n) \in H^n$ 
tels qu'il existe $H'\in \mathsf D_H$ et une valuation  $w \in \mathscr V\cap \zr L{(K,v)}$ à
valeurs dans $H$ 
satisfaisant les conditions $w(f_i) = h_i$ pour tout~$i$.

\begin{enumerate}
\item Le sous-foncteur $D_{(K,v),L,\mathscr V ,(f_i)}$ de $\Gamma_{v(K^\times)}^n$
est polyédral et de dimension majorée par~$d$.
\item Si $\mathscr V = \P_L$ alors $D_{(K,v),L,\mathscr V ,(f_i)}$
est purement de dimension~$d$, et si de plus $n = d$
alors $D_{(K,v),L,\mathscr V ,(f_i)} = \Gamma_{v(K^\times)^n}$.
\end{enumerate}

\item Il existe un recouvrement fini $(\mathscr U_j)$
de $\P_K$ par des parties constructibles et, pour tout $j$, 
une famille finie $(\phi_{jk\ell})_{k,\ell}$ de monômes
de la forme $\prod_i x_i^{r_i}$ où les $r_i$ sont rationnels, et une famille $(a_{jk \ell})_{k,\ell}$ 
d'éléments de $K^\times$ tels que 
pour tout $v\in \mathscr U_j$ le sous-foncteur 
polyédral
$D_{(K,v),L,\mathscr V ,(f_i)} $ de
$\Gamma_{v(K^\times)^n}$
 soit décrit par la condition 
 \[\bigvee_k\bigwedge_\ell v(a_{jk\ell})\phi_{jk\ell}\leq 1.\]
 \end{enumerate}
\end{theo}

\begin{proof}
Nous allons en gros reprendre \textit{mutatis mutandis}
la preuve du théorème 1.3 de \cite{ducros2012b}, qui établissait
essentiellement (1) lorsque
$v$ est la valuation triviale. 
Il résulte des définitions que
 $D_{(K,v),L,\mathscr V_1\cup \mathscr V_2 ,(f_i)}$
 est égal à $D_{(K,v),L,\mathscr V_1 ,(f_i)} \cup D_{(K,v),L,\mathscr V_2 ,(f_i)} $
 pour tout couple $(\mathscr V_1,\mathscr V_2)$ d'ouverts quasi-compacts de $\P_L$. 
 Il suffit donc de démontrer le théorème pour chacun des termes d'un recouvrement fini de $\mathscr V$
 par des ouverts affines, ce qui permet de supposer
 $\mathscr V$ affine. Choisissons une partie finie $S$
 de $L^\times$ telle que $\mathscr V=\P_L\{S\}$, que l'on prend
 égale à l'ensemble vide si $\mathscr V=\P_L$.

La première étape consiste à se réduire au cas non gradué.
Pour ce faire, choisissons un polyrayon 
$r$ dont les composantes engendrent $\abs{L^\times}$. 
Posons $K'=K(T/r)$ et $L'=L(T/r)$. 
Pour toute valuation $v$ de $K$, notons 
$v'$ l'extension de Gauss de $v$ à $K'$, 
caractérisée par la formule
$v'(\sum a_I T^I)=\max v(a_I)$ pour tout élément homogène
$\sum a_I T^I$ de l'anneau gradué $K[T/r]$
(dont $K'$ est le corps gradué des fractions).
On a alors par construction $v'((K')^\times)=v(K^\times)$. Nous noterons
de même $w'$ l'extension de Gauss de $w$ à $L'$ pour toute valuation $w$ de $L$. 
La restriction de $v'$ à $K'_1$ sera notée $v'_1$ et celle de $w'$ à $L'_1$ sera notée $w'_1$. 

Par choix de $r$,
il existe pour tout $\lambda$ 
dans $(L')^\times$  un multi-exposant $e(\lambda)$ tel que 
$\abs{T^{e(\lambda)}\lambda}=1$. 
Si $v$ est une valuation sur $K$ on a 
$v'(T^e)=1$ pour tout $e$ et donc 
$v'(\lambda)=v'(T^{e(\lambda)}\lambda)$
pour tout $\lambda\in K^\times$ ; il vient
$v((K'_1)^\times)=v((K')^\times)=v(K^\times)$. 

Posons $S'=\{T^{e(\lambda)}\lambda\}_{\lambda \in S}$, 
$\mathscr V'=\P_L\{S'\}$ et $\mathscr V'_1=\P_{L'_1}\{S'\}$ ; 
et pour tout $i$, posons $g_i=T^{e(f_i)}f_i$ pour tout $i$. 

Supposons avoir démontré le théorème
pour 
$(K'_1, L'_1,\mathscr V'_1, (g_i))$ et montrons
qu'il vaut pour $(K, L, \mathscr V, (f_i))$.

Soit $v\in \P_K$. Nous allons tout d'abord
démontrer que les deux sous-foncteurs
$D_{(K,v),L,\mathscr V ,(f_i)}$
et $D_{(K'_1,v'_1), L'_1, \mathscr V_1', (g_i)}$
de $\Gamma_{v(K^\times)}^n$ sont égaux. 
Soit $H$ un objet de $\mathsf D_{v(K^\times)}$ et soit $h=(h_1,\ldots, h_n)$ un élément de $H^n$.
Supposons que $h\in D_{(K,v),L,\mathscr V ,(f_i)}(H)$. Il existe 
alors $H'\in \mathsf D_H$ et une valuation  $w$ 
appartenant à $\P_{L/(K,v)}$ à valeurs dans $H'$ telle
que $w(f_i)=h_i$ pour tout $i$. Pour tout $\lambda\in S$ on a $w'(\lambda)=w(\lambda)\leq 1$, et donc
$w'(T^{e(\lambda)}\lambda)\leq 1$. Par conséquent $w'\in \mathscr V'\cap \zr{L'}{(K',v')}$, si 
bien que $w'_1\in  \mathscr V_1'\cap \zr{L'_1}{(K_1',v'_1)}$. Et l'on a pour tout $i$
les égalités $w_1'(g_i)=w'(T^{(e(f_i)}f_i)=w'(f_i)=w(f_i)=h_i$. Par conséquent, $h$ appartient à 
$D_{(K'_1,v'_1), L'_1, \mathscr V'_1, (g_i)}(H)$. 
Réciproquement, supposons que $h$ appartient à $D_{(K'_1,v'_1), L'_1, \mathscr V'_1, (g_i)}(H)$. 
Il existe alors une valuation $\omega$ appartenant à  $\mathscr V_1'\cap \zr{L'_1}{(K_1',v'_1)}$, à valeurs 
dans un groupe $H'\in \mathsf D_H$, telle que $\omega(g_i)=h_i$ pour tout $i$. 
Le morphisme naturel de $L'_1\otimes_{K'_1}K'$ vers $L$ est injectif (c'est même un isomorphisme 
puisque $\abs{(L')^\times}=\abs{(K')^\times}$ par construction), et les restrictions à $K'_1$
de $\omega$ et $v'$ coïncident. Il résulte alors de \cite{ducros2018}, Lemma 7.2.1 qu'il existe
une valuation $\theta$ sur $L'$ telle que $\theta|_{L'_1}=\omega$ et $\theta|_{K'}=v'$. 
Soit $\lambda\in S$. On a 
$\theta(\lambda)=\theta(T^{e(\lambda)}\lambda)$ car $\theta(T^{e(\lambda)})=v'(T^{e(\lambda)})=1$, 
et $\theta(T^{e(\lambda)}\lambda)\leq 1$ car $T^{e(\lambda)}\lambda$ est un élément de $S'$, et que
$\theta|_{L'_1}$ est égale à $\omega$ qui appartient à $\mathscr V'_1$. Par conséquent
$\theta|_L$ appartient à $\mathscr V\cap \zr L{(K,v)}$. 
On a également pour tout $i$ les égalités
$\theta(f_i)=\theta(T^{-e(f_i)}g_i)=\theta(g_i)=h_i$, où la seconde égalité provient du fait que
$\theta(T^{-e(f_i)})=v'(T^{-e(f_i)})=1$. 
Il s'ensuit que $h\in D_{(K,v),L,\mathscr V ,(f_i)}(H)$. 

Puisque le théorème est supposé valoir pour $(K'_1, L'_1,\mathscr V'_1, (g_i))$ 
le sous-foncteur $D_{(K'_1,v'_1), L'_1, \mathscr V_1', (g_i)}$
de $\mathbf G_{v(K^\times)}$ est polyédral ; puisqu'il coïncide avec $D_{(K,v),L,\mathscr V ,(f_i)}$ par
ce qu'on vient de voir, ce dernier est également polyédral. 

Supposons maintenant que $\mathscr V=\P_L$. On a dans ce cas $S=\emptyset$ si bien que
$\mathscr V'_1$
est égal à $\P_{L'_1}$. Puisque le théorème est supposé valoir pour $(K'_1, L'_1,\mathscr V'_1, (g_i))$, 
le sous-foncteur $D_{(K'_1,v'_1), L'_1, \mathscr V_1', (g_i)}$
de $\Gamma_{v(K^\times)}^n$ est purement de dimension $d$, et égal à $\Gamma_{v(K^\times)}^n$
tout entier si $d=n$ ; puisqu'il coïncide avec $D_{(K,v),L,\mathscr V ,(f_i)}$ par
ce qu'on vient de voir, ce dernier est également purement de dimension $d$, et égal à 
$\Gamma_{v(K^\times)}^n$ tout entier si $d=n$. 

Le théorème valant pour $(K'_1, L'_1,\mathscr V'_1, (g_i))$, il existe un recouvrement fini 
$(\Upsilon_j)$ de $\P_{K'_1}$ par des parties constructibles, et une famille finie 
$(\psi_{jkl\ell})_{k,\ell}$ de monômes de la forme
$\prod_i x_i^{r_i}$ où les $r_i$ sont rationnels et une famille $(b_{jk\ell})_{k,\ell}$
d'éléments de $(K'_1)^\times$ tels que pour tout $\theta\in \Upsilon_j$ le sous-foncteur
polyédral  $D_{(K'_1,\theta), L'_1, \mathscr V_1', (g_i)}$
de $\Gamma_{\theta((K'_1)^\times)}^n$ 
 soit décrit par la condition 
 \[\bigvee_k\bigwedge_\ell \theta(b_{jk\ell})\psi_{jk\ell}\leq 1.\]
 
Soit $f\in K'_1$. Elle s'écrit $\sum a_I T^I/\sum b_I T^I$ ; si $v\in \P_K$, on a alors
$v'_1(f)\leq 1$ si et seulement si $\max v(a_I)\leq \max v(b_I)$. Il s'ensuit que l'image réciproque
par $v\mapsto v'_1$ de toute partie constructible de $\P_{K'_1}$ est une partie constructible 
de $\P_K$. Pour tout $j$, notons alors $\mathscr U_j$ l'image réciproque de $\Upsilon_j$ sous 
$v\mapsto v'_1$ ; les $\mathscr U_j$ constituent un recouvrement fini de $\P_K$ par des parties constructibles. 

Fixons $j$. Soit $v\in \mathscr U_j$. On sait que
$D_{(K,v),L,\mathscr V ,(f_i)}$ est égal à
$D_{(K'_1,v'_1), L'_1, \mathscr V_1', (g_i)}$. Puisque $v'_1$ appartient à $\Upsilon_j$
par choix de $v$, on en déduit que $D_{(K,v),L,\mathscr V ,(f_i)}$ est décrit par 
la condition  \[\bigvee_k\bigwedge_\ell v'_1(b_{jk\ell})\psi_{jk\ell}\leq 1.\]
Chacune des $b_{jk\ell}$ s'écrit sous la forme $\sum c_{jk\ell, I}T^I/\sum d_{jk\ell, I}T^I$ et 
l'on a alors \[v'_1(b_{jk\ell})=\frac{\max_I v(c_{jk\ell, I})}{\max_I v(d_{jk\ell, I})}.\]
Pour tout $J,J'$ tels que $c_{jk\ell ,J}$ et $d_{jk\ell ,J'}$ soient non nuls
notons $\mathscr U_{j,J,J'}$ l'ensemble des $v$ appartenant à 
telles que $\max_I v(c_{jk\ell, I})=v(c_{jk\ell,J'})$ et  $\max_I v(d_{jk\ell, I})=v(d_{jk\ell, J''})$. 
Alors $\P_K$ est recouvert par les  $\mathscr U_{j,J,J'}$ qui sont constructibles, et 
pour tout $v\in \mathscr U_{j,J,J'}$ le foncteur  $D_{(K,v),L,\mathscr V ,(f_i)}$ est décrit par 
la condition  \[\bigvee_k\bigwedge_\ell v(c_{jk\ell,J}/d_{jk\ell, J'})\psi_{jk\ell}\leq 1,\]
ce qui montre (2). 

Il suffit donc bien de démontrer le théorème pour $(K'_1,L'_1, \mathscr V'_1, (g_i))$. Autrement dit, 
on peut supposer que $K$ et $L$ sont trivialement gradués, ce que nous faisons désormais. 

Soit~$F$ le sous-corps de~$L$ engendré par~$K$, les~$f_i$ et $S$ et soit~$\mathscr  W$ l'ouvert quasi-compact
$\P_F\{S\}$ de $\P_F$ ; notons que si $S=\emptyset$; c'est-à-dire si $\mathscr V=\P_L$,
alors $\mathscr W=\P_F$. 
L'ouvert $\mathscr V$ est égal à l'image réciproque de $\mathscr W$
sur $\P_L$. Soit $v\in \P_K$ et soit $H$ un objet de $\mathsf D_{v(K^\times)}$. 
Comme toute valuation sur~$F$ s'étend en une valuation sur $L$, les foncteurs 
$D_{(K,v),L,\mathscr V, (f_i)}$ et $D_{(K,v),F,\mathscr  W,(f_i)}$.
sont égaux. Il suffit donc de savoir démontrer
le théorème pour $(K,F,\mathscr W, (f_i))$. Autrement dit on s'est ramené au cas où $L$ est de type fini sur $K$.

Il existe alors un~$K$-schéma intègre et de type fini~$\mathscr  X$ dont le corps des fonctions est isomorphe à~$L$, sur lequel
les éléments de $S$ et les~$f_i$ sont définies, et sur lequel les~$f_i$ sont inversibles. 

Soit~$v$ une valuation sur $K$. Soit $X_v$ le foncteur qui associe à une extension non trivialement valuée et
algébriquement close $(F,w)$ de $(K,v)$
le sous-ensemble de~$\mathscr  X(F)$ formé des points~$P$ tels que
$w(\lambda(P))\leq 1$ pour tout $\lambda \in S$. 
L'assertion A) du théorème du théorème 1.2 de \cite{ducros2012b} assure l'existence d'un sous-foncteur 
polyédral $\Delta_v$ de~$\Gamma_{v(K^\times)}^n$, de dimension~$\leq d$, tel que
$(w\circ f_1,\ldots, w\circ f_n)(X_v(F))=\Delta_v(w(F^\times))$ 
pour tout $(F,w)$, et son assertion (C2) assure que si $S=\emptyset$ (auquel cas $X_v=\mathscr X$) 
le foncteur polyédral $\Delta_v$ est purement de dimension $d$, et égal
à $\Gamma_{v(K^\times)}^n$
tout entier
si l'on a de surcroît $d=n$. 
L'assertion (B) de \loccit  assure par ailleurs que
$\Delta_v=D_{(K,v), L, \mathscr V, (f_i)}$, ce qui achève de prouver (1). 

Il reste à démontrer (2). Soit $E$ l'ensemble des familles finies $(a_{k\ell}, \phi_{k\ell})_{k,\ell}$ 
où les $\phi_{k,\ell}$ sont des monômes de la forme $\prod x_i^{r_i}$ avec les $r_i$ rationnels, 
et où les $a_{k\ell}$ appartiennent à $K^\times$. 
Pour toute $\Phi=(a_{k\ell}, \phi_{k\ell})_{k,\ell}$ appartenant à $E$,
notons $\mathscr U_\Phi$
l'ensemble des valuations $v$ appartenant à $\P_K$ telles que
$D_{(K,v), L, \mathscr V, (f_i)}$ puisse être décrit par la condition 
\[\bigvee_k\bigwedge_\ell v(a_{k\ell})\phi_{k\ell}\leq 1.\]
Il résulte de l'élimination des quantificateurs dans la théorie {\sc acvf}
que $\mathscr U_\Phi$ est une partie constructible de $\P_K$. 
L'énoncé (1) que nous avons établi assure que les $\mathscr U_\Phi$ recouvrent $\P_K$ ; 
par compacité constructible, c'est déjà le cas d'un nombre fini d'entre eux, ce qui achève de montrer (2). 
\end{proof}

\section{Fibrés vectoriels 
sur un espace de Zariski-Riemann}

On désigne toujours par $K$ un corps gradué et par $F$
une extension graduée de $K$.

\subsection{}
En associant à un ouvert $U$ de $\zr LK$ le sous-anneau gradué 
$\mathscr O(U)$ de $L$ dont la composante $\mathscr O(U)^r$ est constituée
des $f\in L_r$ telles que $v(f)\leq 1$ pour toute $v\in U$, on fait de
$\zr LK$ un espace localement annelé gradué ; la fibre $\mathscr O_v$ de son
faisceau structural en une valuation $v$ est précisément l'anneau gradué de~$v$. 
La notion de fibré vectoriel (gradué) sur un ouvert~$U$
de $\zr LK$ a donc un sens
(\ref{ss.fibrevect-gradue}). Si $U$ est non vide, il est connexe
et non vide 
(puisqu'il contient le point générique $\eta$
de l'espace irréductible $\zr LK$), donc
le rang d'un fibré vectoriel $E$ sur $U$ est bien défini. C'est en
particulier la dimension du $L$-espace vectoriel gradué $E_\eta$.

\subsection{}\label{ss-fibre-strict}
Supposons que l'ouvert $U$ est strict. Nous dirons qu'un fibré vectoriel $E$ sur $U$ de rang 
$n$ est strict s'il existe un 
recouvrement ouvert $(U_i)$ de $U$ par des ouverts stricts et pour tout $i$ un isomorphisme
entre $E|_{U_i}$ et $\mathscr O_{U_i}^n$. 

Soit $\pi$ l'application continue naturelle de $\zr LK$ vers $\zr {L^1}{K^1}$. Posons
$U^1=\pi(U)$. On a $U=\pi^{-1}(U^1)$ et $\mathscr O_{U^1}\simeq (\pi_*\mathscr O_U)^1$. 
Les foncteurs $E\mapsto (\pi_*E)^1$ et $F\mapsto \pi^{-1}(F)\otimes_{\pi^{-1}(
\mathscr O_{U^1})}\mathscr O_U$ 
établissent alors une équivalence entre la catégorie des fibrés vectoriels stricts de rang $n$ sur $U$ et celle des
fibrés vectoriels de rang $n$ au sens classique sur $U^1$, c'est-à-dire celle des $\mathscr O_{U^1}$-modules
localement libres de rang $n$.

\begin{rema}
Supposons que $U$ est strict, et soit $E$ un fibré vectoriel de rang $n$ sur $U$. Supposons qu'il existe un recouvrement 
$(U_i)$ de $U$ par des ouverts stricts et, pour tout $i$, un isomorphisme
$E|_{U_i}\simeq \mathscr O_{U_i}[r_i]$ où chacun des $r_i$ est de la forme $(r_{ij})_j$ 
avec les $r_{ij}$ appartenant à $K^\times$.  Le fibré vectoriel $E$ est alors strict. Pour le voir, 
on choisit pour tout $(i,j)$ un élément $a_{ij}$ de $K^\times$ tel 
que $\abs{a_{ij}}=r_{ij}$, et l'on remarque que $(x_1,\ldots, x_n)\mapsto
(a_{i1}x_1,\ldots, a_{in}x_n)$ établit un isomorphisme de
$\mathscr O_U[r_i]$ sur $\mathscr O_U^n$ pour tout $i$. 
\end{rema}

\subsection{}\label{def-metrique-graduee}
Soit $K$ un corps
gradué, soit $n$ un entier, soit $L$ une extension graduée
de $K$ et soit $U$ un ouvert de $\zr LK$. Soit $F$ un $L$-espace vectoriel gradué. 
On appelle métrique sur $F$ au-dessus de $U$ la donnée, pour tout élément homogène $s$ de $F$, 
d'une application $v\mapsto v(s)$ de source $U$ et possédant les propriétés  suivantes : 
\begin{itemize}
\item $v(s)$ appartient à $v(L^\times)$ si $s$ est non nulle, et vaut $0$ sinon ; 
\item on a $v(s+t)\leq \max (v(s),v(t))$ pour tout $v$ et tout couple $(s,t)$ d'éléments de $F$
homogènes de même degré ; 
\item on a $v(as)=v(a)v(s)$ pour tout $s$ et tout élément homogène $a$ de $L$. 
\end{itemize}

Supposons donnée une métrique sur $F$
au-dessus de $U$. Soit $(e_1,\ldots, e_n)$ une base du $L$-espace vectoriel gradué $F$ et soit $V$ un ouvert de $U$. 
Nous dirons que $(e_1,\ldots, e_n)$ est orthonormée au-dessus de $V$ (relativement à la métrique choisie)
si $v(\sum a_i e_i)=\sup_i v(a_i)$ pour toute famille $(a_i)$ d'éléments homogènes de $L$ et tout $v\in V$. 

Nous dirons que la métrique considérée admet localement sur $U$ une base orthonormée si pour tout $v$ de $U$
il existe un voisinage ouvert $V$ de $v$ dans $U$ et une base de $F$ qui est orthonormée au-dessus $V$. 

Soit $F$ et $F'$ deux $L$-espaces vectoriels gradués, chacun étant muni d'une métrique au-dessus de $U$. 
On appelle isométrie de $F$ sur $F'$
au-dessus de $U$ (ou simplement isométrie de $F$ sur $F'$ s'il n'y a pas d'ambiguïté sur $U$)
tout isomorphisme $u$ de $L$-espaces vectoriels gradués tel que $v(u(s))=v(s)$ pour tout $v\in U$ et tout élément
homogène $s$ de $F$. 

\subsection{}
Soit $E$ un fibré vectoriel
de rang $n$ sur un ouvert
non vide $U$
de $\zr LK$. 
Soit $s$ un élément homogène du $L$-espace vectoriel gradué 
$E_\eta$. 
Soit $v\in U$. Il existe un voisinage ouvert 
$V$ de $v$ dans $U$
et un polyrayon  $r=(r_1,\ldots, r_n)$ tels
que $E|_V$
soit isomorphe à $\mathscr O_V[r]$. Le fibré $E|_V$ possède donc une
base $(e_1,\ldots, e_n)$ formés d'éléments homogènes de $E(V)$, chaque $e_i$
étant de degré $r_i$. On peut écrire
d'une unique manière $s=\sum a_i e_i$, où chaque $a_i$ est un élément homogène de
$L$. L'élément $\sup_i v(a_i)$ de $v(L^\times)\cup\{0\}$ ne dépend alors 
que de $v$ et de $s$, et pas du choix de $V$, de $r$ ni de $(e_i)$. 
En effet, il suffit pour le voir de démontrer que l'ensemble des $\lambda\in L^\times$ tels que
$\sup_i v(a_i)\leq v(\lambda)$ ne dépend que de $v$ et de $s$. Mais cet ensemble est aussi celui des
$\lambda\in v(L^\times)$ tels que $v(\lambda^{-1}a_i)\leq 1$ pour tout $i$, c'est-à-dire tels que 
$a\lambda^{-1}a_i\in \mathscr O_v$ pour tout $i$, c'est-à-dire encore tels que $s\in E_v$. 
Il est dès lors licite de poser $v(s)=\sup_i v(a_i)$. 

La collection pour $s$ variable des applications $v\mapsto v(s)$ définit une métrique au-dessus de $U$ sur le $L$-espace vectoriel
gradué $E_\eta$ ; par construction, cette métrique possède une base orthonormée localement sur $U$.

\subsection{}
La construction ci-dessus définit un foncteur de la catégorie des fibrés vectoriels gradués de rang $n$ sur $U$ vers celle des
$L$-espaces vectoriels gradués de dimension $n$ munis d'une métrique au-dessus de $U$ admettant localement sur $U$ une base
orthonormée. 

Ce foncteur est une équivalence de catégories. On en construit un quasi-inverse comme suit. À un 
$L$-espace vectoriel gradué $M$ de dimension $n$ muni d'une métrique au-dessus de $U$ admettant localement sur $U$ une base
orthonormée, on associe le sous-faisceau $F$ du faisceau constant $\underline M$ sur $U$
dont l'ensemble des sections homogènes de degré $r$ sur un ouvert non vide $V$
de $U$ est $\{s\in M_r\;\text{tq.}\;\forall v\in V, v(s)\leq 1\}$. 

Si $(e_1,\ldots, e_n)$ est une base de $M$ et si $V$ est un ouvert de $U$, la base $(e_1,\ldots, e_n)$ est orthonormée au-dessus de
$U$ si et seulement si les $e_i|_V$ constituent une base de $F|_V$. 

\begin{rema}
Soit $(U_i)$ un recouvrement ouvert de $U$, soit $M$ un $L$-espace
vectoriel gradué de dimension finie et pour tout $i$, soit $(g_{ij})_j$
une base de $M$. Il existe alors au plus un fibré vectoriel $E$ sur $U$
tel que $E_\eta=M$ et tel que les $(g_{ij})$ forment une base
de $E$ sur $U_i$ pour tout $i$. En effet si un tel $E$ existe et si l'on 
munit $M$ de la métrique correspondante au-dessus de $U$ on a alors
$v(s)=\sup_j v(a_j)$ dès que $s$ est une section homogène de $M$ s'écrivant
$\sum a_j g_{ij}$ et dès que $v\in U_i$, ce qui suffit à reconstituer entièrement
la métrique de $M$ au-dessus de $U$ puisque les $U_i$ recouvrent $U$. 
\end{rema}
\subsection{}\label{metrique-graduee-extension}
Cette construction
est compatible au changement de corps, dans le sens suivant. Soit $L'$ un sous-corps gradué de 
$L$ et soit $K'$ un sous-corps gradué de $L'\cap K'$. Soit $\pi$ l'application naturelle de 
$\zr LK$ vers $\zr {L'}{K'}$. Supposons qu'il existe un ouvert $U'$ de $\zr {L'}{K'}$ et un fibré vectoriel $E'$
sur $U'$ tels que $U=\pi^{-1}(U')$ et $E=\pi^*E'$. 
On a alors $E_\eta=E'_\eta\otimes_{L'}L$ et en considérant un recouvrement de $U'$ par des 
ouverts sur lesquels $E'$ possède une base, on voit que si $t$ est une section homogène non nulle
de $E'$ et si $v$ est une valuation sur $U$ d'image $v'$ sur $U'$ alors
$v(t)=v'(t)$.

\begin{exem}
Soit $E$ et $F$ deux fibrés vectoriels gradués sur $U$. Le produit tensoriel $E\otimes F$ se décrit ainsi en termes
de $L$-espace vectoriel gradué et de métrique. 

\begin{itemize}
\item Le $L$-espace vectoriel gradué sous-jacent à $E\otimes F$ est $E_\eta \otimes_L  F_\eta$. 
\item La métrique sur $E_\eta\otimes_L F_\eta$ définie par $E\otimes_L F$ se décrit comme suit, à partir
des métriques sur $E_\eta$ et $F_\eta$ définies par $E$ et $F$ : pour tout $v\in U$ et tout élément homogène
$s$ de $E_\eta\otimes_L F_\eta$, on a \[v(s)=\inf_{s=\sum_i t_i\otimes r_i} \sup_i v(t_i)v(r_i).\]

\end{itemize}
\end{exem}

\begin{exem}
Soit $E$ un fibré vectoriel gradué sur $U$. Le dual 
$E^\vee$ se décrit ainsi en termes $L$-espace vectoriel gradué et de métrique. 

\begin{itemize}
\item Le $L$-espace vectoriel gradué sous-jacent à $E^\vee$ est le dual
$E_\eta^\vee$ du $L$-espace vectoriel gradué $E_\eta$. 

\item La métrique sur $E_\eta^\vee$ définie par $E^\vee$ se décrit comme suit, à partir
de la métrique sur $E_\eta$ définie par $E$ : pour tout $v\in U$ et tout élément homogène
$s$ de $E_\eta^\vee$, l'élément $v(s)$ de $v(L^\times)\cup\{0\}$ 
est caractérisé par le fait que 
$v(s(t))=v(s)v(t)$ pour toute section homogène $t$ de $E_\eta$.

\end{itemize}
\end{exem}

\begin{lemm}\label{lemm-critere-def-fibre}
Soit $M$ un sous-corps gradué de $L$ contenant $K$, soit $U$
un ouvert non vide de $\zr LK$ et soit $E$ un fibré vectoriel sur $U$.
Soit $\pi\colon \zr LK\to \zr MK$ l'application canonique.
Les assertions suivantes sont équivalentes : 
\begin{itemize}
\item[(i)]il existe un ouvert $V$ de $\zr MK$ et un fibré vectoriel $F$
sur $V$ tels que $U=\pi^{-1}(V)$ et $E\simeq \pi^*F$ ; 

\item[(ii)] il existe un $M$-espace vectoriel gradué $G$, un isomorphisme 
$G\otimes_M L\simeq E_\eta$, et une famille $(V_i)$ d'ouverts 
de $\zr MK$ tels que : 
\begin{itemize}
\item[1)] l'ouvert $U$ est la réunion des $\pi^{-1}(V_i)$ ; 
\item[2)] pour tout indice 
$i$ il existe une famille finie $(g_{ij})_j$ d'éléments de $G$
dont l'image dans $E_\eta=G\otimes_ML$ est une base de $E$ sur $\pi^{-1}(U_i)$. 
\end{itemize}
\end{itemize}

De plus si elles sont satisfaites on peut choisir $F$
comme dans (i) de sorte que $F_\eta=G$ et que $(g_{ij})_j$ 
soit pour tout $i$ une base de $F$ sur $V_i$. 
\end{lemm}

\begin{proof}
Si (i) est vraie alors (ii) est vraie en prenant $G=F_\eta$, en prenant pour $(V_i)$
un recouvrement ouvert de $V$ trivialisant $F$, et en se donnant pour chaque $V_i$ une base 
de $F$ sur $V_i$. 

Supposons réciproquement que (ii) soit vrai, et utilisons l'isomorphisme donné entre 
$G\otimes_M L$ et $E_\eta$ pour voir $G$ comme contenu dans $E_\eta$. 
Puisque les $\pi^{-1}(V_i)$ recouvrent $U$ et puisque $\pi$ est surjective, 
les $V_i$ recouvrent $V$ et $U=\pi^{-1}(V)$. 
Pour tout $i$, notons $(g_{ij})$ une famille d'éléments de $G$ constituant une base
de $E$ sur $\pi^{-1}(V_i)$. La métrique sur $E_\eta$ est telle que, pour tout $i$ et toute famille
$(a_j)$ d'éléments homogènes de $L$ tels que les $a_je_{ij}$ soient tous de même degré, on ait
$v(\sum a_j {ij})=\max v(a_j) p$ $v\in \pi^{-1}(V_i)$. Si les $a_j$ appartiennent tous à $M$ on 
voit que cette valeur ne dépend que de la restriction de $v$ à $M$, c'est-à-dire que de son image
sur $V$, et appartient à $v(M^\times)$. Les $\pi^{-1}(V_i)$ recouvrant $U$,
il en résulte que pour tout $g\in G^\times$ et 
tout $v\in U$, l'élément $v(g)$ de $v(L^\times)$ ne dépend que de la restriction 
$w$ de $v$ à $M$, et appartient à $w(L^\times)$. 
On a ainsi défini une métrique $g\mapsto (w\mapsto w(g))$ sur $G$.
Par construction, pour tout $i$ la famille $(g_{ij})$ est orthonormée au-dessus de
$V_i$ pour cette métrique. Elle définit donc un fibré vectoriel $F$ au-dessus de $V$, 
de fibre générique $G$, admettant $(g_{ij})_j$ pour base sur chacun des $V_i$. 

Si $W$ est un ouvert de $V$, une section homogène de $F$ sur $W$
est un élément homogène $t$ de $G$ tel que $w(t)\leq 1$ pour tout $w\in W$. 
Lorsqu'on voit $t$ comme un élément homogène de $E_\eta$ il vérifie l'inégalité
$v(t)\leq 1$ pour tout $v\in \pi^{-1}(W)$, et appartient donc à $E(\pi^{-1}(W))$. 
On vient ainsi de définir un morphisme de $F$ vers $\pi^*E$, soit encore
par adjonction un morphisme de $\pi^*F$ vers $E$. Comme ce morphisme envoie
pour tout $i$ la base $(g_{ij})$ de $F|_{V_i}$ sur la base $(g_{i})$ de $E|_{\pi^{-1}(V_i)}$,
c'est un isomorphisme. 
\end{proof}

\begin{defi}
Soit $M$ un sous-corps gradué de $L$ contenant $K$, soit $U$
un ouvert non vide de $\zr LK$ et soit $E$ un fibré vectoriel sur $U$.
Si les conditions équivalentes du lemme ci-dessus sont satisfaites,
nous dirons que $M$ est un \textit{corps gradué de définition} 
du couple $(U,E)$. 
\end{defi}

Lorsque $(E_\ell)$ est une famille de fibrés vectoriels sur $U$,
nous dirons que
$M$ est un corps gradué de définition de $(U,(E_\ell))$, si c'est un corps
gradué de définition de chacun des $(U,E_\ell)$. 

\begin{rema}\label{rem-corps-gradue-def}
Si $U$ est quasi-compact, toute famille finie $(E_\ell)$ de fibrés vectoriels sur $U$ possède 
un corps gradué de définition 
qui est de type fini sur $K$. En effet, fixons pour chaque $\ell$ un isomorphisme
entre $E_{\ell,\eta}$ et $L^{n_\ell}$ (où $n_\ell$ est le rang de $E_\ell$). 
Choisissons un recouvrement fini $(U_i)$ de $U$ où chaque $U_i$
est un ouvert affine $\zr LK\{S_i\}$ pour un certain sous-ensemble fini $S_i$
de $L^\times$, et où chacun des $E_\ell$ possède une base 
$(a_{i,\ell,j})_{1\leq j\leq n_\ell}$ sur $U_i$, chacun des 
$a_{i,\ell,j}$ étant lui-même un $n_\ell$-uplet d'éléments homogènes 
de $L$. 
Le sous-corps gradué de $L$ engendré par $K$, les $S_i$ et les composantes
des différents $a_{i,\ell,j}$ est alors de type fini sur $K$,
et est un corps gradué de définition de $(U,(E_\ell)_\ell)$. 
\end{rema}

\begin{lemm}\label{lemme-puissance-p}
Supposons $K$ de caractéristique $p>0$. Soit $E$ un fibré en droites sur $U$
et soit $M$ un corps gradué de définition de $(U,E)$. Soit $N$ un sous-corps
gradué de $M$ contenant $K$ et contenant $M^q$ pour une certaine puissance 
$q$ de $E$. Alors $N$ est un corps gradué de définition de $(U,E^{\otimes q})$.
\end{lemm}

\begin{proof}
Le choix d'un élément homogène arbitraire non nul de $E_\eta$ permet
d'identifier ce dernier à $L$. Dans ce
cas $E^{\otimes}q_\eta$ s'identifie encore à $L$, 
mais sa métrique est telle que $v_{E^{\otimes q}}(s^q)=v_E(s)^q$ pour tout élément
homogène $s$ de $L$ et tout $v\in U$, en rappelant en indice de la valuation 
le fibré vectoriel en jeu. 
Soit $\pi$ l'application naturelle de 
$\zr LK$ vers $\zr MK$ et soit
$\rho$ celle de $\zr MK$ vers $\zr NK$. Soit $V$ l'image de $U$
sur $\zr MK$. Par hypothèse et en vertu 
du lemme \ref{lemm-critere-def-fibre}
il existe un recouvrement ouvert $(V_i)$ de $V$
et pour tout $i$ un élément $s_i$ de $M^\times$ qui 
est une base de $E$ sur $\pi^{-1}(V_i)$. Quitte à raffine
le recouvrement $(V_i)$, on peut supposer que $V_i$ est pour tout $i$
de la forme $\zr MK\{S_i\}$ pour un certain sous-ensemble fini $S_i$
de $M^\times$. Mais on a alors également 
$V_i=\zr MK\{S_i^q\}$, et $S_i^q\subset N$. 
Par conséquent,
chacun des $V_i$ est égal à $\rho{-1}(W_i)$ pour un certain ouvert quasi-compact
$W_i$ de $\zr NK$. Pour tout $i$ et tout $v$ appartenant à 
$(\rho\circ \pi)^{-1}(W_i)=U_i$ on a $v_E(s_i)=1$ et donc
$v_{E^{\otimes q}})(s_i^q)=1$. Comme les $s_i^q$ appartiennent tous
à $N$ par hypothèse, le lemme \ref{lemm-critere-def-fibre}
assure que $N$ est un corps gradué de définition de $(U,E^{\otimes q})$. 
\end{proof}

\begin{lemm}\label{lemm-secglob-max}
Soit $U$ un ouvert quasi-compact 
de $\zr LK$ et soit $E$ un fibré en droites sur $U$. 
Soit $s$ une section homogène non nulle de $E_\eta$
et soit $(a_i)$ une famille finie
non vide d'éléments de $L^\times$. 
Les assertions suivantes sont équivalentes : 

\begin{enumerate}
\item Les $a_is$ sont des sections globales de $E$ qui engendrent ce dernier
au-dessus de $U$. 
\item On a
$v(s)^{-1}=\max_i v(a_i)$ pour tout
$v\in U$. 
\end{enumerate}
\end{lemm}
\begin{proof}
Posons $s_i=a_is$. Supposons que les $s_i$ sont des sections globales de $E$
sur $U$ et qu'elles l'engendrent. Soit $v\in U$. 
On a 
$v(s)^{-1}=v(a_i)v(s_i)^{-1}$
pour tout $i$. 
Comme les $s_i$ sont des sections de $E$, 
on a pour tout $i$
l'inégalité $v(s_i)\leq 1$, qui entraîne que  $v(s)^{-1}\geq v(a_i)$. Et 
puisque les $s_i$ engendrent $E$ sur $U$, il existe $j$ tel que 
$v(s_j)=1$, ce qui entraîne que $v(s)^{-1}=v(a_j)$. Par conséquent
$v(s)^{-1}=\max_i v(a_i)$. 

Réciproquement supposons que $v(s)^{-1}=\max_i v(a_i)$. 
Soit $v\in U$. On a alors pour tout $j$ les égalités
$v(s_j)=v(a_j)v(s)=v(a_i)/(\max_i v(a_i))\leq 1$, 
et $v(s_j)=1$ pour tout $j$ tel que $v(a_j)=\max_i v(a_i)$. 
Ceci entraîne que les $s_j$ sont des sections globales de $E$ sur $U$
et qu'elles l'engendrent. 
\end{proof}

\begin{coro}\label{coro-engendre-globalement-gradue}
Soit $U$ un ouvert quasi-compact 
de $\zr LK$ et soit $E$ un fibré en droites sur $U$. 
Soit $s$ une section homogène non nulle de $E_\eta$. Le fibré $E$
est engendré par ses sections globales si et seulement s'il existe
une famille finie non vide $(a_i)\in L^\times$ telle que
$v(s)^{-1}=\max_i v(a_i)$ pour tout $v\in U$. 
\end{coro}

\begin{rema}\label{rem-ai-sur-aj}
Supposons qu'il existe un sous-corps gradué $L'$ de $L$, un sous-corps gradué 
$K'$ de $L'\cap K$, un ouvert $U'$ de $\zr{L'}{K'}$ et un fibré en droites $E'$ sur
$U'$ tel que $U=\pi^{-1}(U')$ et $E=\pi^* E'$ où $\pi\colon \zr LK\to \zr{L'}{K'}$ est l'application naturelle.
Supposons que $E'$ soit engendré par sections globales et soit $t$ un élément homogène non nul de
$E'_\eta$. Il existe alors une famille finie $(b_i)$ d'éléments de $(L')^\times$ tels que 
$v(t)^{-1}=\max_i v(b_i)$ pour toute $v\in U'$. Mais cette égalité reste vraie pour toute $v\in U$
lorsqu'on voit $t$ comme une section homogène non nulle de $E_\eta=E'_\eta\otimes_{L'}L$
(\ref{metrique-graduee-extension}). Soit $s$ une section homog!ne non nulle quelconque 
de $E_\eta$. On a alors 
\[v(s)=v((s/t)t)=\max_i v((s/t)b_i)=\max_i v(a_i),\]
en posant $a_i=(s/t)b_i$. Et pour tout $(i,j)$, le quotient 
$a_i/a_j$ est égal à $b_i/b_j$ et appartient donc à $(L')^\times$. 
\end{rema}

\begin{lemm}\label{lemme-existe-maxv(ai)}
Soit $E$ une droite vectorielle graduée et soit $t$ un élément homogène non nul de $E$. 

\begin{enumerate}
\item L'unique métrique sur $E$ telle que $v(t)^{-1}=\max_i v(a_i)$ pour tout $v\in \zr LK$
est induite par un fibré en droites $D$ sur $\zr LK$ de fibre générique $E$. 
\item Pour tout $j$, le sous-corps gradué $K(a_i/a_j)_i$ de $L$ est un corps gradué de définition 
de $D$. 
\end{enumerate}
\end{lemm}

\begin{proof}
Montrons tout d'abord \emph a). 
Pour tout $j$, notons $U_j$ l'ouvert $\zr LK\{a_i/a_j\}_i$ de $\zr LK$, et 
$t_j$ l'élément homogène $a_jt$ de $E$. 
On a alors pour tout $v\in U_j$  les égalités
\[v(t_j)=v(a_j)v(t)=v(a_j)/(\max_i v(a_i))=1\] par définition de $U_j$. 
Ainsi $E$ possède-t-il localement une base orthonormée (ce qui,en dimension $1$, signifie simplement un 
vecteur homogène de norme constante $1$) sur $\zr L K$. Par conséquent, la métrique donnée sur $E$
est bien induite par un fibré en droites
$D$ sur $\zr LK$ dont $E$ est la fibre générique. 

Montrons maintenant \emph b). Soit $\rho$ le degré de $t_j$.
Le choix de $t_j$ comme base induit une isométrie de $E$ sur
$L[\rho]$, où ce dernier est muni de l'unique métrique pour laquelle
$v(1)=v(a_j)/\max_i(v(a_i))$, soit encore $v(1)^{-1}=\max_i v(a_i/a_j)$. 
Soit $F$ le sous-corps gradué $K(a_i/a_j)_i$ de $L$. 
Par l'assertion a) déjà traitée, la métrique sur $F[\rho]$ au dessus de 
$\zr FK$ telle que $v(1)^{-1}=\max_i v(a_i/a_j)$ est induite par un fibré en droites
$\Delta$ sur $\zr FK$ de fibre générique $F[\rho]$. L'image réciproque de $\Delta$
sur $\zr LK$ est alors un fibré en droites de fibre générique $L[\rho]$, dont la métrique
induite sur $L[\rho]$ est telle que $v(1)^{-1}=\max_i v(a_i/a_j)$. Par conséquent, 
l'image réciproque de $\Delta$ sur $\zr LK$ est isomorphe à $D$. 
\end{proof}

\section{Nombres d'intersection sur un espace de Zariski-Riemann}
Soit $K$ un corps gradué et soit $L$ une extension de $K$.

\subsection{Modèles}
Un \emph{pré-modèle}\index{pré-modèle} (resp. un \emph{modèle}\index{modèle}) 
de~$L$ sur~$K$ est un $K$-schéma gradué de type fini, séparé et intègre~$V$ 
muni d'un $K$-plongement $K(V)\hookrightarrow L$ 
(resp. d'un $K$-isomorphisme $K(V)\simeq L$). Sauf mention expresse
du contraire, un morphisme dominant entre pré-modèles sera toujours
supposé compatible avec les plongements de leurs corps des fonctions
dans~$L$. Modulo cette convention, il existe au plus un morphisme
dominant entre deux pré-modèles de $L$ sur $K$. 
Si $W$ et $V$ sont deux
pré-modèles de~$L$ sur~$K$, on dira que $W$ domine
$V$ (resp. est propre sur~$V$) pour
signifier qu'il existe un morphisme dominant (resp. un morphisme
propre et surjectif)
$W\to V$.

Si $V$ est un pré-modèle de $L$ sur $K$, on note $\zr LK\{V\}$ le
sous-ensemble de $\zr LK$ formé des valuations dont l'anneau domine
un anneau local de $V$. C'est un ouvert quasi-compact de $\zr LK$
: pour s'en assurer, on peut supposer que $V$ est affine, auquel
cas $\zr LK\{V\}=\zr LK\{f_1,\ldots,f_r\}$ pour n'importe quelle
famille génératrice homogèhe $(f_1,\ldots,f_n)$ de la $K$-algèbre 
graduée $\mathscr
O_V(V)$. L'application naturelle $\zr LK\{V\}\to V$ est surjective (tout
anneau local gradué de corps des fractions gradué
contenu dans $F$ est dominé par
un anneau de valuation gradué de $F$).

Si $W$ et $V$ sont deux pré-modèles de $F$ sur $K$ et si $W\to V$
est un morphisme dominant, alors $\zr LK\{W\}\subset \zr LK\{V\}$, et
$\zr LK\{W\}$ est égal
à $\zr LK\{V\}$ si et seulement si $W\to V$ est propre (c'est une simple
traduction du critère valuatif) ; en particulier
(prendre $V=\Spec K$) on a $\zr LK \{W\}=\zr LK$ si et seulement
si $W$ est propre.

Soit $V$ un pré-modèle de~$L$ sur $K$.
L'application~$\pi$ de~$\zr LK\{V\}$ dans~$ V$ qui applique
une valuation~$v$ sur son centre (dans~$V$) est continue;
en effet, l'image réciproque de tout ouvert~$U$ de~$V$
est l'ouvert~$\zr LK\{U\}$ de~$\zr LK$.
En outre, un élément
homogène $f\in\mathscr O_V(U)$ induit tautologiquement
un élément~$\pi^*(f)$ de~$\mathscr O_{\zr LK}(\pi^{-1}(U))$; si $f$
s'annule en un point~$x$ de~$U$, $\pi^*(f)$ appartient
à l'idéal maximal de toute valuation~$v$ qui domine~$x$.
Ainsi, $\pi$ est un morphisme d'espaces localement annelés
gradués.

\begin{lemm}\label{lemm-eclatement-indetermination}
Soit $V$ un $K$-schéma gradué intègre et séparé
muni d'un plongement de
$K(V)$ dans $L$, soit $r>0$ et soit 
$f$ un élément homogène de degré 
$r$ de $L$. 
Il existe un sous-schéma intègre et fermé
$W$ de $\P^{(1,r)}_K\times_K V$ muni d'un morphisme
propre et surjectif sur $V$, 
et un $K(V)$-plongement de $K(W)$ dans $L$
qui identifie $K(W)$ à $K(V)(f)$, 
tels que pour tout $w\in W$
ou bien $f\in \mathscr O_{W,w}$ ou bien $f^{-1}\in \mathscr
O_{W,w}$. 
\end{lemm}

\begin{proof}
Soit $U$ un ouvert affine non vide de $V$ ; posons
$A=\mathscr O(U)$, et soit $B$ le sous-anneau gradué intègre 
$A[f]$ de $L$. 
Le morphisme surjectif $A[T]\to B, T\mapsto f$ fait 
de $\Spec B$ un sous-schéma gradué fermé de 
$U\times_K \A^r_K$, où $\A^r_K=\Spec K[T/r]$ ; 
le morphisme de $\Spec B$ vers 
$\A^r_K$ défini par $f$ est par construction le morphisme
induit par la seconde projection de $\A^r_K\times_K U$. 

En immergeant  $\A^r_K$ dans $\P^{(1,r)}_K$ de manière standard, on fait de $B$
un sous-schéma gradué de $U\times_K \P^{1,r}_K$, puis de
$V\times_K \P^{(1,r)}_K$. 
Soit $W$ l'adhérence réduite de $\Spec B$ dans $V\times_K \P^{(1,r)}_K$. 
C'est un $K$-schéma gradué intègre. La première projection définit un morphisme
propre de $W$ sur $V$, qui est dominant puisque $W$ contient $\Spec B$
comme ouvert dense, et qui est donc surjectif. La seconde projection définit
un morphisme de $W$ vers $ \P^{1,r}-K$ qui prolonge le morphisme $f$
de $\Spec B$ vers $\A^r_K$. L'existence même de ce prolongement
signifie précisément que $f\in \mathscr O_{W,w}$ ou bien $f^{-1}\in \mathscr
O_{W,w}$ pour tout 
$w\in W$. 

\end{proof}

\begin{lemm}[Variante graduée du lemme de Chow]
Soit $V$ un pré-modèle de~$L$ sur~$K$.
Soit 
$W$ un pré-modèle de $L$ dominant $V$. Il existe 
un pré-modèle $W'$ de $L$ propre et birationnel
sur $W$ et quasi-projectif sur $V$. 
\end{lemm}

\begin{proof}
Soit $p$ l'unique morphisme 
dominant de $W$ sur $V$
induisant l'inclusion $K(V)\subset K(W)$ de sous-corps
gradués de $L$. 
Soit $(V_i)$ et $(W_i)$ deux recouvrements de 
$V$ et $W$, respectivement, par des ouverts affines
denses tels que $p(W_i)\subset V_i$ pour tout $i$. 
On peut écrire pour tout $i$ l'anneau gradué
$\mathscr O(W_i)$ comme un quotient de $\mathscr O_{V_i}[T/r_i]$ 
pour un certain polyrayon $r_i$, ce qui permet de voir 
$W_i$ comme un sous-schéma gradué fermé de $V_i\times_K \A^{r_i}_K$, 
puis comme un sous-schéma gradué de $V_i\times_K \P^{1,r_i}_K$, puis de
 $V\times_K \P^{1,r_i}_K$. Soit $Z_i$ l'adhérence réduite de $W_i$ dans 
$V\times_K \P^{1,r_i}_K$ et soit $U$ l'intersection de tous les $W_i$
dans $W$.  
Pour tout $i$, on dispose par construction d'une immersion 
de $W_i$ dans $Z_i$, d'où un morphisme de $U$ dans 
$\prod_K Z_i$ ; soit $Z$
l'adhérence réduite de $U$ dans $\prod_K Z_i$. 
Pour tout $i$, notons $W'_i$ l'image réciproque sur $Z$
de $W_i$ par la projection sur $Z_i$, où $W_i$ est vu
comme un ouvert
dense de $Z_i$.

Soit $W'$ la réunion des $W'_i$. 
Le schéma gradué intègre $W'$ contient $U$ comme ouvert dense, 
et a donc $K(W)$ comme corps gradué des fonctions. Il 
s'immerge par construction
dans $V\times_K \prod_i \P^{1,r_i}_K$,
puis dans $V\times_K \P^s_K$ pour un certain
polyrayon $s$ à l'aide d'une variante graduée du plongement de Segre. 
Par construction, chacun des $W'_i$ domine (en tant que schéma gradué
intègre de corps gradué des fonctions $K(W)$) l'ouvert $W_i$ de $W$, si bien
que $W'$ domine $W$.

Il suffit pour conclure de montrer que $W'$ est propre
sur $W$.
Chacun des $W'_i$ est par construction propre sur $W_i$, puisque $W'_i\to W_i$ est 
obtenue par changement de base à partir d'un morphisme entre deux 
$V$-schémas gradués propres. Nous allons montrer que l'image réciproque
de $W_i$ sur $W'$ par le morphisme $f\colon W'\to W$ est égale à $W'_i$,
ce qui terminera
la démonstration puisque la propreté est une propriété locale sur le but. 
La définition même de $W_i$ et de $f$ entraîne que $W'_i\subset f^{-1}(W_i)$. 
Or comme $W'_i\to W_i$ est propre et $f^{-1}(W_i)\to W_i$
séparé, l'image de $W'_i$ dans $f^{-1}(W_i)$ est un fermé
de $f^{-1}(W_i)$. 
Puisque $W'_i$ est un ouvert dense de $W'$ et 
a fortiori de $f^{-1}(W_i)$, il vient $W'_i=f^{-1}(W_i)$. 
\end{proof}

\begin{coro}\label{coro-quasiproj-cofinal}
Pour tout pré-modèle $W$ de $L$ sur $K$, 
il existe un pré-modèle $W'$ de $L$ sur $K$ propre et birationnel
sur $V$ et quasi-projectif sur $K$.
\end{coro}

\begin{proof}
Il suffit d'appliquer le lemme précédent avec $V=\Spec K$.
\end{proof}

\begin{prop}
Soit $V$ un pré-modèle de $L$ sur $K$. 
Soit $\pi$ le morphisme 
d'espaces localement annelés gradués de
$\zr LK\{V\}$
vers la limite de tous les pré-modèles de $L$ sur $K$ propres sur $V$. 
Le morphisme $\pi$ est un isomorphisme.
\end{prop}

\begin{proof}
Montrons d'abord que $\pi$ est ensemblisement bijectif. 
Soit $u$ et $v$ deux valuations de $\zr LK\{V\}$ telles que 
$\pi(u)=\pi(v)$. Soit $f\in L^\times$. En vertu du lemme 
\ref{lemm-eclatement-indetermination}, il exuste un pré-modèle $W$
de $L$ propre sur $V$ tel que pour tout $w\in W$ on ait $f\in \mathscr 
O_{w,W}$ ou $f^{—1}\in\mathscr O_{W,w}$. 
Par hypothèse, $u$ et $v$ ont même image $w$ sur $W$. 
Si $f\in \mathscr O_{W,w}$ alors $f\in \mathscr O_u$ et $f\in \mathscr O_v$ ; 
dans le cas contraire $f^{-1}$ appartient à l'idéal maximal de
$\mathscr O_{W,w}$, et donc à l'idéal maximal de $\mathscr O_u$ et à celui
de $\mathscr O_v$, et il n'appartient dès lors ni à $\mathscr O_u$
ni à $\mathscr O_v$. Il s'ensuit que $\mathscr O_u=\mathscr O_v$, c'est-à-dire
que $u=v$, et $\pi$ est injective. 

Montrons que $\pi$ est bijective. Supposons donné une famille compatibles
$(x_W)_W$ 
de 
points sur tous les pré-modèles de $L$ propres sur $V$. Il résulte du lemme
\ref{lemm-eclatement-indetermination} que la colimite $R$ des anneaux locaux
de ces points contient $f$ ou $f^{-1}$ pour tout $f\in L^\times$. Par conséquent,
$R$ est l'anneau gradué d'une valuation de $L$ ; par construction, 
$R$ domine tous les $\mathscr O_{W,x_W}$. C'est donc l'anneau
gradué d'une valuation $u$
de $L$ qui par construction est centrée sur $V$ (puisque $R$ domine
$\mathscr O_{V,x_V}$) et est envoyée par $\pi$ sur le système des $x_W$.

Montrons que $\pi$ est un homéomorphisme. Soit $E$ un sous-ensemble fini de 
$L^\times$. Toujours d'après le lemme \ref{lemm-eclatement-indetermination}
il existe un pré-modèle $W$ de $L$ propre sur $V$ et tel que pour tout $(f,w)\in 
E\times W$ on ait $f\in \mathscr O_{W,w}$ ou $f^{-1}\in \mathscr O_{W,w}$. 
Soit $W'$ l'ouvert de $W$ constitué des
points $w$ tels que $f\in \mathscr O_{W,w}$
pour toute $f\in E$. Alors $\zr LK\{V\}\cap \zr LK\{E\}$ est l'image
réciproque de $W'$ sur $\zr LK \{V\}$. Comme les ouverts de $\zr LK\{V\}$ 
de la forme  $\zr LK\{V\}\cap \zr LK\{E\}$ (pour $E$ variable) engendrent
la topologie de $\zr LK\{V\}$, la bijection continue $\pi$ est ouverte, et donc
est un homéomorphisme. 

Enfin, le raisonnement suivi pour établir la surjectivité de $\pi$ 
montre que si $u\in \zr LK\{V\}$ et si 
$(x_W)_W$ désigne l'image $\pi(u)$ alors
l'anneau gradué local $\mathscr O_u$ 
est la colimite des anneaux gradués locaux 
$\mathscr O_{W,x_W}$, ce qui montre que $\pi$
est un isomorphisme d'espaces localement annelés gradués. 
\end{proof}

\begin{defi} Soit $U$ un ouvert quasi-compact de $\zr LK$, soit $d$
un entier  et
soit $E$ un fibré vectoriel de rang $d$ sur $U$.
On appellera
\textit{pré-modèle} (resp. \textit{modèle}) de $E$ tout couple $(V,F)$ où $V$
est un pré-modèle (resp. un modèle) de $L$ sur $K$ tel que $U=\zr
LK\{V\}$ et où $F$ un fibré vectoriel sur $V$ dont l'image inverse sur $U=\zr LK\{V\}$ est isomorphe à $E$.
\end{defi}

Si $(V,F)$ est un pré-modèle d'un fibré vectoriel $E$ sur un
ouvert quasi-compact $U$ de $\zr LK$ alors pour tout
pré-modèle $W$ de $L$ sur $K$ qui est propre sur $V$, le couple $(W,F|_V)$
est un pré-modèle de $E$. 

\subsection{}\label{cocycles-niveau-fini}
Soit $E$ un fibré vectoriel de rang $d$ sur $U$. Choisissons
un recouvrement ouvert fini 
$(U_i)$ de $U$ où $U_i=\zr LK\{S_i\}$ pour un certain sous-ensemble 
fini $S_i$ de $L^\times$ et pour tout $i$ un isomorphisme
$\phi_i\colon E|_{U_i}
\simeq \mathscr O_{U_i}[r_i]$ pour un certain polyrayon $r_i$. Le fibré vectoriel
$E_i$ peut alors être reconstitué à isomorphisme près par le «cocycle»
$(\phi_i|_{U_i\cap U_j}\circ \phi_j^{-1}|_{U_i\cap U_j})_{i,j}$, qu'on peut
voir comme un cocyle matriciel. L'ensemble des ces données mettant en jeu un 
ensemble fini d'éléments de $L^\times$, il existe un pré-modèle
$(V,F)$ de $(U,E)$. Plus généralement, on peut construire un 
pré-modèle $(V, F_i)_i$ pour toute famille finie $(E_i)$ de fibrés vectoriels
sur $U$, en construisant d'abord un pré-modèle pour chacun des $E_i$, puis en 
prenant un pré-modèle qui les domine tous. 

Si $L$ est de type fini sur $K$, tout fibré vectoriel sur $U$, et plus
généralement toute famille finie de fibrés vectoriels sur $U$, admet un 
modèle.

Si $E$ est in fibré en droites sur $U$, il est engendré par ses sections
globales si et seulement s'il possède un pré-modèle $(V,F)$ où $F$
est engendré par ses sections globales. 

\begin{rema}
Supposons que $L$ soit de type fini et degré de transcendance 1 sur $K$. C'est
alors
comme dans le cas non gradué
le corps gradué des fonctions d'une $K$-courbe
graduée projective intègre et
normale $C$ qui est uniquement déterminée, et la flèche naturelle $\rho \colon
\zr LK\to C$ est un isomorphisme d'espaces localement annelés gradués ; tout
fibré vectoriel $E$
sur un ouvert quasi-compact $U$ de $\zr LK$ 
admet dès lors un modèle canonique, 
à savoir le couple $(\rho(U),\rho_*E)$. 
\end{rema}

\begin{lemm}\label{lemme-difference-globalement-engendre}
Soit $U$ un ouvert quasi-compact de $\zr LK$ et soit $E$ un fibré en droites
sur $U$. Il existe deux fibrés en droites $D_1$ et $D_2$ sur $U$ engendrés par
leurs sections globales et un isomorphisme $E\simeq D_1\otimes D_2^\vee$. 
\end{lemm}

\begin{proof}
Choisissons un pré-modèle $(V,G)$ de $(U,E)$. D'après le corollaire 
\ref{coro-quasiproj-cofinal} on peut supposer $V$ quasi-projectif. 
Il résulte alors de \ref{ss.quasiproj-difference}
qu'il existe deux fibrés en droites $F_1$ et $F_2$ sur $V$, engendré par leurs
sections globales, et tels que $G\simeq F_1\otimes F_2^\vee$. 
Il suffit pour conclure de prendre pour $D_1$ et $D_2$ les images réciproques 
respectives de $F_1$ et $F_2$ sur $U$. 
\end{proof}

\begin{exem}\label{exemple-o1-gradue}
Soit $D$ une droite vectorielle graduée sur $L$ et soit $t$ un élément homogène non nul de $D$, 
de degré $r$. Soit $(a_i)_{1\leq i\leq n}$ une famille finie d'éléments de $L^\times$ algébriquement indépendants 
sur $K$ et soit $E$ le fibré en droites sur $\zr LK$ dont la fibre générique est $D$ et dont la métrique 
est telle que $v(t^{-1})=\max(1, v(a_1),\ldots, v(a_n))$ (lemme \ref{lemme-existe-maxv(ai)}). 

Identifions $\P^{(1,\abs{a_1},\ldots, \abs{a_n})}_K$ à un prémodèle de $L$ sur $K$ via le plongement 
$K(T_i/T_0)_{1\leq i\leq n}\hookrightarrow L, (T_i/T_0)\mapsto a_i$ dont nous notons $F$ l'image. 
Soit $\Delta$ l'image réciproque sur $\zr FK$ du fibré en droites $\mathscr O(1)[r]$ de
$\P^{1,\abs{a_1},\ldots, \abs{a_n}}_K$. La fibre générique
de $\Delta$ est égale à $F\cdot T_0$ où $T_0$
est considéré comme homogène de degré $r$. Les éléments $T_0$ et $T_i=a_i T_0$ constituent une 
famille de sections globales de $\Delta$ qui engendrent ce dernier. Il résulte alors du lemme
\ref{lemm-secglob-max} que la métrique sur $F\cdot T_0$ au-dessus de $\zr FK$ induite par
$\Delta$ est telle que $v(T_0)^{-1}=\max( 1,v(a_1),\ldots v(a_n))$ pour tout $v$. 
Il s'ensuit que l'image réciproque de $\Delta$ sur $\zr LK$ est isomorphe à $E$, c'est-à-dire encore que 
 $(\P^{(1,\abs{a_1},\ldots, \abs{a_n})}_K,\mathscr O(1)[r])$ est un pré-modèle de $(\zr LK, E)$. C'en est un modèle
 si $L=K(a_i)_{1\leq i\leq n}$. 
\end{exem}

\subsection{}\label{ss-intersection-nongraduee}
Soit $K\hookrightarrow L$ une extension de corps
classiques gradués ; soit 
$n$ un entier
supérieur ou égal au degré de transcendance de
$L$ sur $K$. 
Soit $(E_1,\dots, E_n)$ une famille de
fibrés en droites sur $\zr LK$.
Il existe un  pré-modèle propre $(V, (F_i))$
de $(\zr LK, (E_i))$
(on peut même si on le souhaite prendre pour $V$
un modèle propre de $L$ sur $K$, mais ce n'est pas nécessaire pour ce qui suit). 
La formule de projection (lemme \ref{proprietes-intersection-graduee}
\emph d))
entraîne que l'entier défini par
\[ \begin{cases} [L:K(V)]\, \int_V c_1(L_1)\cdots c_1(L_n) & 
\text{si $\dim(V)=n$;}\\
0 &\text{si $\dim(V)<n$} \end{cases}\]
ne dépend pas du choix de~$V$ et des~$F_i$. 
On le notera
\[ \int_{\zr LK} c_1(E_1)\dots c_1( E_n). \]

\subsection{}\label{ss-intersection-changebasefini}
Soit $L'$ une extension finie de $L$ et soit $\pi$ le morphisme 
naturel de $\zr{L'}K$ vers $\zr LK$. Pour tout $i$, notons $E'_i$ 
le fibré en droites $\pi^*E_i$ sur $\zr{L'}K$.
Soit $(V, (F_i))$ un modèle propre de la famille des
$L_i$. C'est alors également un pré-modèle propre de la famille
des $E'_i$, si bien qu'on a les égalités 
\begin{align*}
 \int_{\zr {L'}K}c_1(E'_1)\dots c_1( E'_n)&=
[L':L] \int_V c_1(F_1)\cdots c_1(F_n)
\\
 &=[L':L]\int_{\zr LK} c_1(E_1)\dots c_1( E_n). 
 \end{align*}

\subsection{}\label{coro-intersection-graduee}
Soit $K'$ une extension 
de $K$.
Choisissons un modèle propre $(V, (F_i))$ de la famille des $E_i$. 
Le $K'$-schéma $V'=K'\otimes_k L$ possède alors $r$ composantes irréductibles 
$V_1,\ldots, V_r$
(considérées comme munies de la structure réduite), 
qui sont toutes de dimension $n$ (on se ramène à l'énoncé non gradué analogue grâce au lemme
\ref{lemm-composantes-graduees}). 
Notons $\mathscr O_j$ 
local de $V'$ au point générique de $V_j$, notons $m_j$ sa longueur
et $L_j$ son corps gradué résiduel, c'est-à-dire le corps gradué des fonctions de 
$V_j$. Remarquons que $K'\otimes_K L$ est noethérien (en écrivant $L$ comme corps gradué 
des fractions de $\mathscr O_V(U)$ où $U$ est un ouvert affine non vide de $V$, on voit en effet que $K'\otimes_K L$
est un localisé d'un anneau gradué noethérien) et que les $\mathscr O_j$ sont ses localisés
en ses différents idéaux premiers gradués minimaux. Les $\mathscr O_j$, les $m_j$ et les $L_j$
ne dépendent donc que de $K'$ et pas du choix du modèle $V$. Pour tout $(i,j)$, notons 
$E_{ij}$ l'image réciproque de $E_i$ sur $\zr{L_j}{K'}$.

Pour tout $i$ et tout $j$ notons 
$F'_i$ l'image réciproque de $F_i$ sur $V'$
et $F_{ij}$ son image réciproque sur $V_j$.
Pour tout indice $j$ la famille $(V_j, (F_{ij})_i)$ est alors un modèle propre
de la famille $(E_{ij})_i$. 
Il vient 
\begin{align*}
\int_{\zr LK} c_1(E_1)\dots c_1( E_n)&=\int_V c_1(F_1)\cdots c_1(F_n)\\
&=\int_{V'}c_1(F'_1)\cdots c_1(F'_n)\\
&=\sum_j m_j \int_{V_j}c_1(F_{1j})\cdots c_1(F_{nj})\\
&=\sum_j m_j \int_{\zr {L_j}{K'}}c_1(E_{1j})\dots c_1( E_{nj}).
,\qedhere
\end{align*}
où la troisième égalité résulte du lemme \ref{proprietes-intersection-graduee}
\emph c).

Mentionnons un cas particulier important : si $K'\otimes_K L$ est intègre, si l'on désigne par
$L'$ son corps gradué des fractions et par
$E'_i$ l'image réciproque de $E_i$ sur $\zr{L'}{K'}$ on a alors
\[\int_{\zr LK} c_1(E_1)\dots c_1( E_n)=\int_{\zr{L'}{K'}}c_1(E'_1)\dots c_1( E'_n).\]

\subsection{}\label{ss-numeff-nongradue}
Soit $(E_1,\ldots, E_m)$ une famille finie
de fibrés en droites sur $\zr LK$ et soit $(V,F_1,\ldots, F_m)$ un pré-modèle de $(\zr LK, E_1,\ldots, E_m)$  ;
remarquons que $V$ est propre. Nous dirons que $c_1(E_1)\ldots c_1(E_m)$ est numériquement effectif si 
$c_1(F_1)\ldots, c_1(F_m)$ est numériquement effectif (\ref{def-numeff}), 
ce qui ne dépend pas du choix du pré-modèle $V$. Nous dirons qu'un fibré en droites
$E$ sur $\zr LK$ est numériquement effectif si $c_1(E)$ est numériquement effectif. Tout fibré en droites
sur $\zr LK$ engendré par ses sections globales est numériquement effectif. 

Soit $L'$ une extension de $L$ et soit  $K'$ un sous-corps gradué de $L'$ contenant $K$. 
Pour tout $i$ notons $E'_i$ l'image réciproque de $E_i$ sur $\zr{L'}{K'}$. 
Si $c_1(E_1)\ldots c_1(E_m)$ est numériquement effectif il en va de même de 
$c_1(E'_1)\ldots c_1(E'_m)$ : cela résulte de la préservation  du caractère numériquement effectif 
par morphismes propres et extension des scalaires (voir \ref{def-numeff}).

\begin{exem}\label{exemple-o1-ZR}
Soit $D$ une droite vectorielle graduée sur $L$ et soit $t$ un élément homogène non nul de $D$.
Soit $(a_i)_{1\leq i\leq n}$ une famille finie d'éléments de $L^\times$ algébriquement indépendants 
sur $K$ telle que $K=K(a_i)_i$ et soit $E$ le fibré en droites sur $\zr LK$ dont la fibre générique est $D$ et dont la métrique 
est telle que $v(t^{-1})=\max(1, v(a_1),\ldots, v(a_n))$ (lemme \ref{lemme-existe-maxv(ai)}). 
L'exemples \ref{exemple-o1-gradue} exhibe un modèle 
de $(\zr LK, E)$ de la forme $(\P^r_K, \mathscr O(1)[s])$
(pour un certain polyrayon $s$ et un certain $s>0$). 
On déduit alors de l'exemple \ref{exemple-c1-o1} et de la remarque
\ref{rema-intersection-decalage} on voit que $\int_{\zr LK}c_1(E)^n=1$.

\end{exem}

\section{Réduction graduée des anneaux et modules
normés}
Nous allons ici présenter la théorie de la réduction 
graduée des anneaux et modules normés telle qu'introduite
par Temkin dans \cite{temkin2004}, en insistant sur le cas
particulier majeur des corps ultramétriques. 

\subsection{Réduction graduée}\label{ss-red-graduee}
Soit $M$ un groupe abélien muni d'une 
semi-norme ultramétrique $\|\cdot\|$, c'est-à-dire
d'une application $\|\cdot\|\colon M\to \R_+$
telle que 
\[\|0\|=0, \quad \|(-m)\|=\|m\|\quad \text{et}\quad \|m+n\|\leq \max(\|m\|, \|n\|)\]
pour tout $(m,n)$. 
On pose
\[\widetilde M=\bigoplus_{r>0}\{m\in M, \|m\|\leq r\}/\{m\in M, \|m\|
<r\}.\]
La loi interne de $M$ induit une loi interne sur $\widetilde M$ qui en fait
un groupe abélien gradué. Dans ce contexte, 
si $\mu$ est un élément homogène de $\widetilde M$ 
nous poserons $\|\mu\|=0$
si $\mu$ est nul, et nous désignerons
sinon par $\|\mu\|$ le degré de l'unique
composante homogène contenant $\mu$ 
(si la semi-norme de $M$ est notée $\abs \cdot$, nous
écrirons $\abs \mu$ plutôt que
$\|\mu\|$). 
Si $m$ est un élément de $M$ et si $r$
est un réel strictement positif tel que 
$r\geq m$, 
on notera $\widetilde m_r$ l'image de $a$ dans $\widetilde
A_r$. 
Nous désignerons par $\widetilde m$ (sans indice)
l'élément homogène de 
$\widetilde M$ égal à $\widetilde m_{\|m\|}$ si $\|m\|$ est non nul, 
et à $0$ sinon. 
Si $N$ est un sous-groupe dense de $M$ pour la
topologie définie par $\|\cdot\|$ il résulte des définitions 
que $\widetilde N=\widetilde M$. .

Soit $A$ un anneau muni d'une 
semi-norme ultramétrique et sous-multiplicative
$\|\cdot\|$.
(nous incluons dans la définition le fait que
$\|1\|\leq 1$, ce qui entraîne que $\|1\|=1$
ou $\|1\|=0$, ce dernier cas n'étant possible que si
$\|\cdot\|$ est identiquement nulle).
La multiplication de $A$ induit
alors une loi interne sur le groupe abélien 
gradué $\widetilde A$ qui en fait un anneau gradué ; nous l'appellerons
\textit{l'anneau gradué résiduel} de $A$.

L'anneau gradué $\widetilde A$ est intègre si et seulement si
$\|\cdot\|$ est multiplicative (nous incluons dans la
définition le fait que $\|1\|=1$). 

Remarquons que $\widetilde A^1$ est l'anneau résiduel
classique $\{a\in A, \|a\|\leq 1\}/\{a\in A, \|a\|<1\}$ de 
$A$, et que l'application $a\mapsto \widetilde a^1$
de $A$ vers $\widetilde A^1$ est l'application de réduction 
classique. 

Si $k$ est un corps ultramétrique complet, 
son anneau gradué résiduel $\widetilde k$ est un 
corps gradué, que nous appellerons le corps
gradué résiduel de $k$.

Soit $M$ un $A$-module muni d'une semi-norme 
$\|\cdot\|$ telle que $\|am\|\leq \|a\|\cdot \|m\|$ pour tout
$(a,m)$. La loi externe de $M$ induit alors une loi de
$\widetilde A\times \widetilde M$ vers $\widetilde M$ qui fait
de ce dernier un $\widetilde A$-module gradué. 

\subsection{}\label{ss-independance-lin-residuelle}
Soit 
$A$ un anneau muni 
une semi-norme sous-multiplicative $\|\cdot\|$, et soit
$M$ un $A$-module muni d'une semi-norme
$\|\cdot\|$ telle que $\|am\|\leq \|a\|\cdot \|m\|$ pour tout $(a,m)$. 
Soit $m=(m_i)$ une famille d'éléments de $A$ tels que $\|m_i\|\neq 0$ pour tout $i$. 

Soit
$(a_i)_{i\in I}$ une famille d'éléments presque tous nuls de $A$.
Supposons que les $\|a_i\|$ sont non toutes nulles, 
soit $r$ le
maximum des $\|a_i\|\cdot \|m_i\|$, 
et soit $I'$ l'ensemble des
indices $i$ tels que  $\|a_i\|\cdot\|m_i\|=r$. Posons $n=
\sum_{i\in I} a_im_i$. 
On a alors $\|n\|\leq r$ et $\widetilde n_r=\sum
_{i\in I'} \widetilde {a_i}\widetilde {m_i}$. 
On a $\|n\|<r$ si et seulement si
$\sum_{i\in I'}\widetilde a_i \widetilde m_i=0$, et
si $\|n\|=r$ alors $\widetilde n=\sum_{i\in I'}\widetilde
a_i \widetilde {m_i}$. 

Il en résulte que les
assertions suivantes sont équivalentes : 

\begin{itemize}
\item[(i)] On a $\|\sum a_i m_i\|=\sup_i \|a_i\|\cdot \|m_i\|$ pour toute famille
$(a_i)$ de scalaires presque tous nuls (notons que 
$\sup_i \|a_i\|\cdot \|m_i\|=\max_i \|a_i\|\cdot \|m_i\|$ dès que $I$ est non vide);
\item[(ii)] La famille $(\widetilde {m_i})_i$ d'éléments homogènes de $\widetilde M$ est libre sur $\widetilde A$. 

Et si elles sont satisfaites alors pour toute famille 
$(a_i)$ d'éléments de $A$ presque tous nuls,
la réduction $\widetilde {\sum a_im_i}$ est égale
à $\sum_{i\in I' }\widetilde a_i \widetilde {m_i}$ où $I'$ désigne l'ensemble des indices
$i$ tels que $\|a_i\|
\cdot \|m_i\|=\sup_j \|a_j\|
\cdot \|m_j\|$. 

\end{itemize}

\subsection{}\label{ss-independance-alg-residuelle}
Soit $k$ un corps ultramétrique complet
et soit $A$ une $k$-algèbre
munie d'une semi-norme multiplicative
$\|\cdot\|$. 
Soit $(a_i)$ une famille d'éléments de
$A$ tels que les $\|a_i\|$ soient tous non nuls. On a
$\|a^J\|=\|a\|^J$ et donc $\widetilde{a^J}=\widetilde a^J$ pour tout multi-exposant $J$. Il résulte alors
de \ref{ss-independance-lin-residuelle}
appliqué à la famille $(a^J)_J$
du $k$-espace vectoriel semi-normé $A$ que les assertions suivantes sont équivalentes : 
\begin{itemize}
\item[(i)] On a $\|\sum \lambda_J a^J\|=\max \abs{\lambda_J}\cdot \|a^J\|$ pour toute famille
$(\lambda_J)$ de scalaires presque tous nuls; 
\item[(ii)] La famille $(\widetilde {a_i})_i$ d'éléments homogènes de $\widetilde A$ est
algébriquement indépendante sur~$\widetilde k$.
\end{itemize}

Et si elles sont satisfaites alors pour toute famille 
$(\lambda_J)$ de scalaires presque tous nuls,
la réduction $\widetilde {\sum \lambda_J a^J}$ est égale
à $\sum_{J\in E}\widetilde \lambda_J \widetilde a^J$ où $E$ désigne l'ensemble des multi-indices
$J$ tels que $\abs{\lambda_J}
\cdot \|a\|^J=\max_{N} \abs{\lambda_N}
\cdot \|a\|^N$.

\begin{exem}\label{exem-ktr-tilde}
Soit $k$ un corps ultramétrique
et soit $r=(r_1,\ldots, r_n)$ un polyrayon de longueur finie. 
Posons $T=(T_1,\ldots, T_n)$ et notons $A$
la $k$-algèbre $k[T]$ munie de la norme 
multiplicative 
$\sum a_I T^I\mapsto \max \abs{a_I}r^I$. 
On déduit de
l'équivalence (i)$\iff$(ii)
de~\ref{ss-independance-alg-residuelle}, appliquée
à la famille~$(T_1,\dots,T_n)$ de $A$, 
que le morphisme 
de
$\widetilde k$-algèbres graduées 
de~$\widetilde k[\tau/r]$ vers~$\widetilde A$
envoyant~$\tau_i$ sur $\widetilde T_i$ est un isomorphisme. 

Soit $K$ le corps des fractions de $A$, muni de la valeur absolue induite par la norme de~$A$. 
Par ce qui précède, il existe un isomorphisme de $\widetilde k$-extensions graduées entre
$\widetilde k(\tau/r)$ et $\widetilde K$, qui envoie $\tau_i$ sur $T_i$. 
\end{exem}

\subsection{}\label{ss-invariants-valuatifs}
Soit $k\hookrightarrow L$ une extension de corps ultramétriques. 
Nous allons voir que quelques invariants classiques se décrivent en un sens plus
simplement en termes de réductions
graduées.

Posons $e=(\abs{L^\times}:\abs{k^\times})$ et
$f=[\widetilde L^1:\widetilde k^1]$ (autrement dit, $e$ et $f$ sont les indices de ramification 
et d'inertie de $L$ sur $k$). Par construction des corps gradués résiduel, 
$\abs{\widetilde k^\times}=\abs{k^\times}$ et $\abs{\widetilde L^\times}=\abs{L^\times}$. 
Il vient 
\[ef=(\abs{\widetilde L^\times}:\abs{\widetilde k^\times})[\widetilde L^1:\widetilde k^1]
=[\widetilde L:\widetilde k],\]
où la seconde égalité provient de \ref{deg-extension-graduee}. 

Soit $(\alpha_i)$ une base de $\widetilde L$ sur $\widetilde k$. Pour tout $i$, choisissons un élément
$a_i$ de $L^\times$ tel que $\widetilde{a_i}=\alpha_i$. Pour toute famille
$(\lambda_i)$ d'éléments de $k$ presque tous nuls on a en vertu de  \ref{ss-independance-lin-residuelle}
l'égalité $\abs{\sum \lambda_i a_i}=\max_i \abs{\lambda_i}\cdot \abs{a_i}$. En particulier, la famille
$(a_i)$ est libre, si bien que $[\widetilde L:\widetilde k]\leq [L:k]$. 

Supposons que $k$ est hensélien (ce sera toujours le cas en pratique dans ce travail,
où nous travaillerons avec des corps
ultramétriques complets) et que $L$ est une extension finie de $k$. 
On appelle \textit{défaut} de $L$ sur $k$ le quotient 
$\frac{[L:k]}{[\widetilde L:\widetilde k]}$, et on dit que l'extension est sans défaut si ce défaut vaut $1$. 
Le défaut est toujours une puissance de l'exposant caractéristique résiduel, et il vaut donc automatiquement $1$
lorsque la caractéristique résiduelle est nulle. Le corps hensélien $k$ est dit \textit{stable} si toutes ses extensions finies
sont sans défaut. Si $k$ est complet et de valuation discrète ou s'il est de caractéristique résiduelle nulle, il est stable. 

Ne supposons plus nécessairement $k$ hensélien, ni $L$ fini sur $k$. 
Soit $(\beta_j)$ une base de transcendance de $\widetilde L$ sur $\widetilde k$. Pour tout $j$, choisissons
un élément
$b_j$ de $L^\times$ tel que $\widetilde{b_j}=\alpha_i$. Posons $b=(b_j)$. Si $(\lambda_J)_J$ est une famille d'éléments de $k$
presque tous nuls paramétrée par les multi-exposants $J$ on a en vertu de  \ref{ss-independance-alg-residuelle}
l'égalité $\abs{\sum \lambda_Jb^J}=\max_i \abs{\lambda_J}\cdot \abs b^J$. En particulier, la famille
$(b_j)$ est algébriquement indépendante, si bien que le degré
de transcendance de $\widetilde L$ sur $\widetilde k$ est inférieur ou égal à celui de $L$ sur $k$. 
Rappelons par ailleurs que l'on a 
\begin{align*}
\mathrm{degtr}(\widetilde L/\widetilde k)&=\mathrm{degtr}(\widetilde L^1/\widetilde k^1)+\dim_\Q(\Q\otimes_\Z\abs{\widetilde L^\times}
/\abs{\widetilde k^\times})\\
&=\mathrm{degtr}(\widetilde L^1/\widetilde k^1)+\dim_\Q(\Q\otimes_\Z\abs{L^\times}
/\abs{k^\times})
\end{align*}
où la première
égalité provient de \ref{deg-extension-graduee}. On obtient ainsi 
l'inégalité d'Abhyankar 
\[\mathrm{degtr}(\widetilde L^1/\widetilde k^1)+\dim_\Q(\Q\otimes_\Z\abs{L^\times}
/\abs{k^\times})\leq \mathrm{degtr}(L/k).\]

Lorsque $L$ est de type fini sur $k$, on dit que $L$ est une extension d'Abhyankar de
$k$ si l'inégalité d'Abhyankar est une égalité.

\chapter{Rappels de géométrie analytique}
\def\pf{\text{\upshape pf}}
\def\red{\text{\upshape red}}
\def\lf{\text{\upshape lf}}
\def\Bor{\text{\upshape Bor}}
\def\loc{\text{\upshape loc}}
\def\Rad{\text{\upshape Rad}}
\def\Zar{\text{\upshape Zar}}
\def\Tau{{\mathrm T}}
\def\gZ{\mathrm Z}
\def\degcal{\operatorname{dc}\nolimits}
\def\codim{\operatorname{codim}}

\section{Géométrie analytique : conventions, notations, premiers rappels}

On fixe pour toute la suite de l'article un corps ultramétrique complet~$k$ ; 
\emph{sa valeur absolue peut être triviale.}
On note $k^\circ$
son anneau de valuation. Nous désignerons par 
$\widetilde k$ le corps \textit{gradué}
résiduel de $k$ (\ref{ss-red-graduee}). Sa composante
$\widetilde k^1$ de degré $1$ est le corps résiduel standard de $k$, 
c'est-à-dire le corps résiduel de l'anneau local $k^\circ$.

Dans ce texte, la notion d'espace~$k$-analytique 
sera à prendre \emph{au sens de Berkovich}, 
et plus précisément de \cite{berkovich1993}. Rappelons qu'un 
polyrayon est 
une famille (le plus souvent finie) de nombres réels strictement positifs  ; 
la \textit{longueur} \index{longueur d'un polyrayon}
d'un polyrayon est le cardinal de son ensemble d'indices. 
Faisons quelques rappels pour la commodité du lecteur. 

\subsection{}\label{ss-description-ma}
La construction fondamentale sur laquelle repose la géométrie de Berkovich est la suivante. 
Soit $A$ un anneau de Banach, c'est-à-dire un anneau $A$ muni d'une norme
sous-multiplicative $\|\cdot\|$ (on inclut dans la définition le fait que
$\|1\|\leq 1$, ce qui entraîne que $\|1\|=1$ dès que $A$ est non nul) pour laquelle il est complet. 
À un tel anneau $A$ on associe son spectre analytique $\mathscr M(A)$, qui est défini comme
l'ensemble des semi-normes multiplicatives bornées sur $A$, c'est-à-dire des applications 
$\phi\colon A\to \R_+$ vérifiant les propriétés suivantes : 
\begin{itemize}
\item on a $\phi(0)=0$ et $\phi(f+g)\leq\phi(f)+\phi(g)$
pour tous $f,g\in\mathscr A$ (propriété de semi-norme) ;
\item on a $\phi(1)=1$ et $\phi(fg)=\phi(f)\phi(g)$ pour tout $f,g\in A$  (multiplicativité);
\item il existe un nombre réel~$c$ tel que $\phi(f)\leq c\norm f$
pour tout $f\in A$ (caractère borné de $\phi$). 
\end{itemize}
Notons qu'on impose que $\phi(0)=0$ et $\phi(1)=1$ : si $A$ est nul, $\mathscr M(A)$
est donc vide. Cet ensemble $\mathscr M(A)$ est muni de la topologie
la moins fine pour laquelle les applications $\phi\mapsto\phi(f)$
sont continues pour tout $f\in A$. L'application $\phi\mapsto (\phi(a))_{a\in A}$
identifie topologiquement $\mathscr M(A)$ à un fermé de l'espace produit $\prod_{a\in A}[0,\|a\|]$, 
si bien que $\mathscr M(A)$ est compact. 
Si $\phi$ appartient à $\mathscr M(A)$, son noyau est un idéal premier de $A$, 
et l'application $\phi \mapsto \ker \phi$ de $\mathscr M(A)$ vers $\Spec A$ est continue. Son image
contient tous les idéaux maximaux de $A$ (c'est ce qui est réellement démontré dans
la preuve du théorème  1.2.1 de \cite{berkovich1990}) ; en particulier $\mathscr M(A)$ est non vide 
dès que $A$ est non nul. 

En général, on note typiquement plutôt $x$ que $\phi$ un élément de $\mathscr M(A)$, en y  pensant comme
à un point. Soit $x\in \mathscr M(A)$. La
semi-norme multiplicative~$x$ induit
une valeur absolue 
sur le corps des fractions 
de l'anneau intègre $\mathscr A/\ker x$,
corps des fractions dont le complété est noté $\hr x$.
L'application composée $\mathscr A\to \mathscr A/\ker x\to \hr x$
sera notée $f\mapsto f(x)$, et la valeur absolue du corps complet
$\hr x$ sera notée $\abs \cdot$ ; avec ces conventions on 
a $x(f)=\abs{f(x)}$ pour toute $f\in A$ et nous utiliserons 
sauf exception la notation de droite. 
Cette construction est fonctorielle : tout morphisme borné $A\to B$ d'anneaux de Banach 
induit par composition une application continue de $\mathscr M(B)$ vers $\mathscr M(A)$.

La semi-norme spectrale de $A$ est la semi-norme
$\rho$ commutant aux puissances définie par les
formules
\[ \rho(f)= \sup_{x\in\mathscr M(A)} \abs{f(x)}
= \liminf_{n\ra\infty} \norm {f^n}^{1/n}
= \inf_{n\ra\infty} \norm{f^n}^{1/n} \]
(voir~\cite{berkovich1990}, 1.3, pour les deux dernières
égalités). 
Nous désignerons par $A^\circ$ 
(resp. $A^{\circ \circ}$) l'ensemble des éléments $f$ de $A$
tels que $\rho(f)\leq 1$ (resp. $\rho(f)<1$), et par $\widetilde A$ l'anneau gradué résiduel de $A$
(\ref{ss-red-graduee}) relativement à $\rho$. Sa composante
$\widetilde A^1$ est l'anneau résiduel standard $A^\circ/A^{\circ \circ}$. 

Supposons maintenant que $A$ est une $k$-algèbre et
que sa norme est une norme de $k$-espace vectoriel. Dans ce cas $\lambda \mapsto \abs{\lambda(x)}$ définit
pour tout $x\in \mathscr M(A)$ une valeur absolue sur $k$ bornée relativement à la valeur absolue structurale de $k$, 
ce qui entraîne que $\abs{\lambda(x)}=\abs \lambda$ pour tout $\lambda$ : la semi-norme $x$
induit donc la valeur absolue structurale de $k$, et $(\hr x, \abs \cdot)$ est une extension de $(k,\cdot)$.
Comme la valeur absolue de $k$
est ultramétrique, celle de $\hr x$ l'est aussi (une valeur absolue $\abs \cdot$
étant ultramétrique si et seulement si 
$\abs n\leq 1$ pour tout $n\in \Z$). On a donc dans ce cas $\abs{(f+g)(x)}\leq \max (\abs{f(x)}, \abs{g(x)})$ pour tout $f,g$
dans $A$,  puis $\rho(f+g)\leq \max (\rho(f), \rho(g))$ (comme on le voit en faisant varier $x$). 

\subsection{}
Soit $n$ un entier. Pour 
tout polyrayon $r=(r_1,\dots,r_n)$,  nous désignerons
par 
$k\{r_1^{-1}T_1,\dots,r_n^{-1}T_n\}$ la sous-algèbre
de $k[[T_1,\dots,T_n]]$ formée des séries formelles
$f=\sum_{\nu\in\N^n} a_\nu T^\nu$ telles que $a_\nu r^\nu = 
a_\nu r_1^{\nu_1}\dots r_n^{\nu_n}$
converge vers~$0$ lorsque $\abs\nu=\sum \abs{\nu_j}$ tend vers~$+\infty$.
On la munit de la norme multiplicative ultramétrique
définie 
par $\norm f=\max \abs{a_\nu}r^\nu$. 
Pour abréger les notations, on notera aussi $k\{r^{-1}T\}$ cette
algèbre. C'est une algèbre de Banach dont tous 
les idéaux sont fermés et de type fini. 

Soit maintenant $A$ une $k$-algèbre de Banach commutative
(par convention, dans cet ouvrage, les normes de $k$-algèbres seront toujours ultramétriques). 
On dit que $A$ 
est une algèbre $k$-affinoïde\index{algèbre $k$-affinoïde}
s'il existe un entier~$n$, un polyrayon $(r_1,\ldots, r_n)$ et une surjection 
de $k\{r^{-1}T\}$ vers~$A$ telle que la norme de~$A$
soit équivalente à la norme quotient correspondante. 
On dit que $A$ est \emph{strictement} $k$-affinoïde
si l'on peut de plus prendre $r_j=1$ pour tout~$j$.\index{algèbre strictement
$k$-affinoïde}
Pour que $A$ soit strictement $k$-affinoïde, il suffit en fait que l'on puisse prendre
les $r_j$ de torsion modulo $\abs{k^\times}$
(\cite{bosch-g-r1984} 6.1.5/4). 

Une $k$-algèbre de Banach $A$ est dite affinoïde s'il existe
une extension valuée complète~$K$ de~$k$ telle que $A$
soit une algèbre $K$-affinoïde.\index{$k$-algèbre affinoïde}

Si $A$ est une $k$-algèbre affinoïde l'application continue naturelle de 
$\mathscr M(A)$ vers $\Spec A$ est surjective (\cite{berkovich1993}, Proposition 2.1.1). 

Soit $A$ une $k$-algèbre affinoïde et soit $\rho$ sa semi-norme spectrale.
L'algèbre $A$ est noethérienne, et ses idéaux sont fermés. 
Pour tout $f\in A$ on a $\rho(f)=0$ si et seulement si $f$ est nilpotent. En particulier, 
$\rho$ est une norme si et seulement si $A$ est réduite et si c'est le cas, $\rho$ est équivalente
à la norme structurale de $A$. 

\begin{exem}\label{exemple-kr}
Soit $r=(r_i)$ un polyrayon de longueur finie. Soit $T=(T_i)$ une famille
d'indéterminées et soit $k_r$ la $k$-algèbre constituée des séries 
formelles $\sum_{I\in \Z^n}a_I T^I$ à coefficients dans $k$, où $\abs{a_I}\cdot r^I$ tend vers zéro quand
$\abs I$ tend vers l'infini, munie de la norme $\sum a_I T^I\mapsto \max \abs{a_I}\cdot r^I$. 
La $k$-algèbre normée $k_r$ est une $k$-algèbre affinoïde : 
elle est en effet isomorphe à $k\{r^{-1}T, rS\}/(ST-1)$. Sa norme est multiplicative
(\cite{ducros2007}, 1.2.1). Supposons de plus que $r$ soit \textit{$k$-libre}, c'est-à-dire 
que les $r_i$ soient $\Q$-linéairement indépendants dans $\Q\otimes_\Z(\R_+^\times/\abs{k^\times})$. 
Dans ce cas $k_r$ est un corps (\cite{ducros2007}, 1.2.2), et c'est donc à la fois une extension complète de $k$
et une algèbre $k$-affinoïde. 
\end{exem}

\begin{rema}
Soit $A=k\{r^{-1}T\}$ une algèbre $k$-affinoïde, où $r=(r_i)$ est un polyrayon de longueur finie. 
Choisissons un ensemble $J$ d'indices
tel que le polyrayon $s:=(r_j)_{j\in J}$ soit $k$-libre et qui soit maximal pour cette propriété. 
Chacun des $r_i$ est alors de torsion modulo $\abs{k^\times}\cdot \prod_{j\in J}r_i^\Z$, 
si bien que la $k_s$-algèbre de Banach $A\widehat{\otimes}_k k_s$ est strictement 
$k_s$-affinoïde. 
\end{rema}

\subsection{}
\label{def.domaff} 
Soit~$\mathscr A$ une algèbre~$k$-affinoïde ; posons~$X=\mathscr M(\mathscr A)$.
Une partie~$V$ de~$X$
est appelée un {\em domaine affinoïde}
de~$X$
si le foncteur qui envoie une $k$-algèbre de Banach affinoïde~$\mathscr B$ 
sur l'ensemble des morphismes bornés de~$\mathscr A$ vers~$\mathscr B$ tels
que l'image de~$\mathscr M(\mathscr B)$ soit contenue dans~$V$ est représentable
par une algèbre~$k$-affinoïde~$A_V$. Si c'est le cas,~$\mathscr M(A_V)\simeq V$.

Un {\em système rationnel d'inéquations}
sur~$X$
est une combinaison booléenne d'inégalités
qui est de la forme
\[\abs{f_1}\leq \lambda_1\abs g\;\text{et}\;\abs{f_2}\leq \lambda_2\abs g\;
\text{et}\;\ldots\;\text{et}\;\abs{f_n}\leq \lambda_ n \abs g\]
où~$g$ et les~$f_i$ sont des fonctions analytiques 
sur~$X$ sans zéro commun, et où les~$\lambda_i$
sont des réels. 

Si~$\beta$ est un système rationnel d'inéquations sur~$X$, on notera~$\mathring\beta $ 
le système
défini à partir de~$\beta$ en remplaçant toutes inégalités larges par les inégalités
strictes correspondantes.

Soit
\[\beta :\abs{f_1}\leq \lambda_1\abs g\;\text{et}\;\abs{f_2}\leq \lambda_2\abs g\;
\text{et}\;\ldots\;\text{et}\;\abs{f_n}\leq \lambda_ n \abs g\]
un système rationnel d'inéquations sur~$X$. 

On qualifiera de \emph{rationnel}
le lieu de~$X$ défini par un tel système d'inéquations.
C'est un domaine 
affinoïde de~$X$, avec
\[A_V\simeq k\{T_1/\lambda_1,\ldots, T_n/\lambda_n\}/(f_i-T_ig)_i.\] 
Notons que si~$g$ s'annulait en un point de~$V$, il en irait de même de chacune des~$f_i$ ; 
en conséquence, $g|_V$ est inversible. 

Tout point de~$X$ possède une base de voisinages rationnels,
en particulier une base de voisinages affinoïdes. Plus précisément, 
pour tout~$x\in X$ et tout voisinage ouvert~$U$ de~$x$, 
il existe un système rationnel d'inéquations~$\beta$
dont le lieu de validité est contenu dans~$U$,
et tel que~$x$ satisfasse~$\mathring \beta$. 

\subsection{}
Le théorème de Gerritzen-Grauert\index{théorème de Gerritzen--Grauert}
affirme que tout domaine affinoïde de~$X$ est réunion 
finie de domaines rationnels
(pour le cas non strict, voir \cite{ducros2003}, lemme~2.4;
pour une preuve directe, voir aussi \cite{temkin2005}, théorème~3.1). 
Il en résulte que si $x$ est un point de $X$ situé sur un domaine affinoïde $V$
de $X$, le corps résiduel complété $\hr x$ ne dépend pas du fait que l'on voie $x$ comme
appartenant à $V=\mathscr M(A_V)$ ou à $X=\mathscr M(A)$. 

\subsection{}
À partir des espaces affinoïdes,
Berkovich définit dans~\cite{berkovich1993}
la catégorie des espaces $k$-analytiques 
par un procédé de recollement un peu délicat qu'il serait trop long de décrire ici
en détail ; nous allons tout de même rappeler quelques faits. Suivant Berkovich, 
si $X$ est un espace topologique  et si $(X_i)$ est une famille de parties de $X$, nous dirons que les
$X_i$ constituent un \textit{G-recouvrement} de $X$
si tout point $x$ de $X$ possède un voisinage de la forme $\bigcup_{i\in J}
X_i$ où $J$ est fini et où $x\in X_i$ pour tout $i\in J$. Si l'on ne s'astreint à recoller que des espaces
strictement affinoïdes, on obtient la classe des espaces strictement $k$-analytiques.

Soit $X$ un espace $k$-analytique. C'est en particulier un espace topologique
dont tout point possède une base de voisinages compacts (mais $X$ n'est pas nécessairement localement compact 
puisqu'un espace topologique localement compact est \textit{par définition} séparé, et que ce n'est 
pas forcément le cas de $X$). L'espace $X$ possède par ailleurs une 
classe distinguée de partie compactes,
ses 
\emph{domaines affinoïdes} (si~$X=\mathscr M(\mathscr A)$,
les domaines affinoïdes de~$X$ sont précisément ceux décrits au~\ref{def.domaff}
ci-dessus). Un \textit{domaine analytique} de $X$ est une partie de $X$ 
qui est G-recouverte par les domaines affinoïdes qu'elle contient. 
Tout domaine affinoïde de $X$ est un domaine analytique de $X$; 
plus généralement, tout compact de $X$ qui est une union finie de domaines 
analytiques de $X$ est un domaine analytique. Et tout ouvert de $X$ (et notamment
$X$ lui-même) est un domaine analytique.  En particulier, tout point $x$ de $X$ possède un voisinage
compact qui est une réunion finie de domaines affinoïdes contenant $x$.
Les espaces définis dans~\cite{berkovich1990} sont ceux qui sont
qualifiés de \emph{bons} dans \cite{berkovich1993}, c'est-à-dire
ceux dont tout point admet un voisinage affinoïde 
(et donc une base de voisinages affinoïdes.)
Tout domaine analytique de $X$ est une partie localement fermée de $X$ : 
l'intersection d'une famille finie de domaines analytiques est un domaine analytique ; 
la réunion d'une famille localement finie de domaines analytiques fermés est un domaine analytique.

Tout domaine analytique $V$ de $X$ hérite d'une structure canonique d'espace analytique, 
et d'un morphisme $V\hookrightarrow X$ dont l'application sous-jacente est l'inclusion. 
Tout morphisme $f\colon Y\to X$ se factorisant ensemblistement par $V$ se factorise en fait uniquement
par la structure analytique de $V$. 

\subsection{La G-topologie}\label{definition-gtopologie}
Soit $X$ un espace $k$-analytique. En plus du site topologique sur $X$ on définit un second site plus fin, noté~$X_\groth$. 
Ses objets sont les domaines analytiques de $X$, ses flèches sont les inclusions, et ses familles
couvrantes sont les G-recouvrements. La topologie de ce site est appelée la G-topologie de $X$. 

Comme tout ouvert de~$X$ est un domaine analytique et comme tout recouvrement ouvert est un G-recouvrement, 
on dispose d'un morphisme de site canonique
$\pi\colon X_\groth\to X$. 
L'unité d'adjonction $\id\to\pi_*\pi^*$ est un isomorphisme
(\cite{berkovich1993}, p.~26).

Le site~$X_\groth$ est annelé. Nous noterons $\mathscr O_X$ le faisceau structural de~$X_\groth$.
Sauf mention du contraire,
lorsque nous parlerons de faisceaux cohérents sur~$X$,
il s'agira plus précisément de $\mathscr O_X$-module cohérents sur~$X_\groth$.
Précisons que le faisceau $\mathscr O_X$ ainsi que sa restriction
au site topologique usuel de $X$ sont cohérents (\cite{ducros2009}, lemme 0.1 ; pour la correction d'une petite erreur
dans la preuve, signalée par Jérôme Poineau, voir \cite{ducros2018}, 1.1.1). 
Si $X=\mathscr M(A)$ alors $\mathscr O_X(X)=A$.

Tout morphisme $f\colon Y\to X$ entre espaces analytiques est en particulier un morphisme de sites annelés : si $V$
est un domaine analytique de $Y$ alors $f^{-1}(V)$ est un domaine analytique de $X$ et l'on dispose d'un 
morphisme $f^*\colon \mathscr O_X(V)\to \mathscr O_Y(f^{-1}(V))$. Si $X$ est un espace $k$-affinoïde 
$\mathscr M(A)$ et si $Y$ est un espace affinoïde $\mathscr M(B)$ sur une extension de $k$, l'application $f\mapsto f^*$ établit une bijection entre l'ensemble des $k$-morphismes
de 
$Y$ vers $X$ et l'ensemble des morphismes bornés de $k$-algèbres de $A$
vers $B$. 

Si $X=\mathscr M(A)$ est un espace $k$-affinoïde et si $M$ est un $A$-module de type fini, le faisceau $\widetilde M$
sur $X_\groth$ associé au préfaisceau $V\mapsto M\otimes_A\mathscr O_X(V)$ est cohérent. Le foncteur 
$M\mapsto\widetilde M$ établit une équivalence entre la catégorie des $A$-modules de type fini et celle des 
faisceaux cohérents sur $X$, de quasi-inverse $\mathscr F\mapsto \mathscr F(X)$
(\cf \cite{berkovich1993} ; ce résultat est dû à Kiehl dans le cas strictement affinoïde).

Lorsque $X$ est bon, 
le foncteur $\pi_*$ induit une équivalence de catégories entre
faisceaux cohérents sur~$X_\groth$ et faisceaux cohérents sur~$X$, 
et qu'un faisceau cohérent $\mathscr F$ sur $X_\groth$ est localement 
libre si et seulement si c'est le cas de $\pi_*\mathscr F$
(\cite{berkovich1993}, 1.3.4 ; notons que la définition de faisceau cohérent
qu'il utilise dans cette proposition est \textit{ad hoc}, mais qu'elle coïncide  bien en fait
avec la notion usuelle au vu de la cohérence des faisceaux structuraux sur les deux sites
annelés en jeu).

\subsection{Produits fibrés, extension des scalaires}
La catégorie des espaces $k$-analytiques admet des produits fibrés. Si $A\to B$ et $A\to C$ sont deux morphismes
d'algèbres $k$-affinoïdes, $\mathscr M(B)\times_{\mathscr M(A)}\mathscr M(C)=\mathscr M(B\widehat \otimes_A C)$. 

Si $L$ est un extension complète de $k$ on dispose d'un foncteur d'extension des scalaires 
$X\mapsto X_L$ de la catégorie des espaces $k$-analytiques vers celle des espaces $L$-analytiques. 
Si $X=\mathscr M(A)$ alors $X_L=\mathscr M(A\widehat \otimes_k L)$. En général on dispose d'un morphisme 
naturel de sites annelés de $X_L$ vers $X$, qui est surjectif (\cf \cite{ducros2007}, 0.5). 

Supposons $X$ bon, soit $y$
un point de $X_L$ et soit $x$  son image sur $X$. Le morphisme naturel 
de $\Spec \mathscr O_{X_L,y}$ vers $\Spec \mathscr O_{X,x}$ est plat à fibres localement d'intersection complète, et
à fibres géométriquement régulières si $L$ est \textit{analytiquement séparable} sur $k$ ; c'est le théorème
\cite{ducros2018}, th. 2.6.5.  La notion d'extension analytiquement séparable est définie dans 
\cite{ducros2009} (définition 1.6). Une extension finie de $k$ est analytiquement séparable si et seulement si elle 
est séparable ; si $r$ est un polyrayon de longueur finie et $k$-libre, $k_r$ est une extension analytiquement séparable de $k$ ; 
si $k$ est parfait, toute extension complète de $k$ est analytiquement séparable.  

On définit la catégorie des espaces analytiques au-dessus de $k$ comme la catégorie dont les objets sout les couples
$(X,L)$ où $L$ est une extension complète de $k$ et $X$ un espace $L$-analytique. Un mrophisme de $(X,L)$ vers $(Y,M)$ est
un couple formé d'un $k$-morphisme de $M$ vers $L$ et d'un morphisme d'espaces $L$-analytiques
de $X$ vers $Y_L$. 

Cette catégorie n'admet pas tous les produits fibrés (par exemple $\mathscr M(\C_p)\times_{\mathscr M(\Q_p)}
\mathscr M(C_p)$ n'existe pas en tant qu'espace analytique au-dessus de $\Q_p$) mais si $Y\to X$ est un morphisme d'espaces
$k$-analytiques et $X'\to X$ un morphisme d'espaces analytiques au-dessus de $k$ alors $Y\times_k X'$ existe : si $L$ est le corps
de définition fourni avec $X'$, le produit fibté $Y\times_k X'$ est simplement le produit fibré $Y_L\times_{X_L}X'$, calculé dans la catégorie
des espaces $L$-analytiques. 

\subsection{}
Soit $X$ un espace $k$-analytique, et soit $x\in X$. 

On note~$\mathscr O_{X,x}$ la fibre en $x$ de la restriction de $\mathscr O_X$
à la catégorie des ouverts de $X$, c'est-à-dire la colimite des $\mathscr O_X(U)$ où $U$
parcourt l'ensemble des voisinages ouverts de $X$. C'est un anneau local 
dont nous noterons $\kappa(x)$
le corps résiduel. 

Lorsque $X$ est bon, 
cet anneau est noethérien (\cite{berkovich1993}, th. 2.1.4), hensélien
(\cite{berkovich1993}, th. 2.1.5) et excellent (\cite{ducros2009}, th. 2.13). 

Au point~$x$ est également associé une extension
ultramétrique complète~$\hr x$ 
de~$k$, qu'on appelle
le \emph{corps résiduel complété} de~$x$ ; pour tout domaine analytique 
$V$ de $X$ contenant $x$, on dispose d'un morphisme d'évaluation 
$f\mapsto f(x)$ de $\mathscr O_X(V)$ dans $\hr x$. 
Si $X=\mathscr M(A)$, le corps $\hr x$ et l'application 
$f\mapsto f(x)$ de $A$ vers $\hr x$ coïncident avec 
ceux définis en \ref{ss-description-ma}. 
Le corps résiduel~$\kappa(x)$ de l'anneau local~$\mathscr O_{X,x}$ 
se plonge naturellement dans~$\hr x$.
On dira que $X$ est \emph{presque bon en~$x$}\index{presque bon!espace ---}
si $\kappa(x)$ est dense dans~$\hr x$
et si le germe~$(X,x)$ est séparé, et qu'il est \emph{presque bon} 
s'il est presque bon en tout point.
Si $X$ est presque bon en~$x$, tout domaine analytique de~$X$
contenant~$x$ est encore presque bon en~$x$.
Un espace bon est presque bon. Le point $x$ est traditionnellement
dit \textit{rigide} si $\hr x$ est une extension finie de $k$ (la terminologie
provient du fait que ces points sont ceux considérés par Tate dans sa théorie des
espaces analytiques rigides). 

Soit $f\colon Y\to X$ un morphisme entre espaces $k$-analytiques et 
soit $x\in X$. La fibre $f^{-1}(x)$ possède une structure naturelle d'espace
$\hr x$-analytique, et sera toujours implicitement considérée comme munie
de cette structure. Nous la noterons souvent $Y_x$. 

\subsection{}
Berkovich a construit (\cite{berkovich1993}, 2.6) un foncteur d'analytification
$\mathscr X\mapsto \mathscr X^\an$ qui va de la catégorie des $A$-schémas
localement de type fini vers celle des bons espaces $k$-analytiques. 
L'espace $\mathscr X^\an$ est muni d'un morphisme d'espaces localement annelés
$\mathscr X^\an\to \mathscr X$ qui est fidèlement
plat (\cite{berkovich1993}, Prop. 2.6.2). L'application $\pi$ induit
une bijection entre
l'ensemble
$\mathscr X_0^{\an}$ des points rigides de $\mathscr X\an$ et l'ensemble
$\mathscr X_0$ des points fermés de $\mathscr X$, qui préserve les corps
résiduels ; on en tire en particulier une bijection 
entre $\{x\in \mathscr X^{\an}, \hr x=k$ et $\mathscr X(k)$, qui est un 
homéomorphisme (la source étant munie de la topologie induite par
celle de $\mathscr X^{\an}$, et le but de la topologie déduite de celle
de $k$). Si $\abs{k^\times}\neq \{1\}$, l'ensemble $\mathscr X^{\an}_0$
est dense dans $\mathscr X^{\an}$. 

Si $x$ est un point de $\mathscr X^\an$ d'image $\xi$ sur $\mathscr X $, le morphisme local
$\mathscr O_{\mathscr X,\xi}\to \mathscr O_{\mathscr X^\an,x}$ est régulier
(\cite{ducros2009}, théorème 3.3). 

Donnons quelques détails sur la construction de $\mathscr X^\an$ lorsque $\mathscr X$ est affine. 
Écrivons $\mathscr X=\Spec A$ où $A$ est une $k$-algèbre de type fini. Topologiquement, 
$\mathscr X^\an$ est l'ensemble des semi-normes multiplicatives sur $A$ qui induisent la valeur
absolue structurale sur $k$, muni de la topologie induite par la topologie produit sur $\R^A$ ; 
le morphisme de $\mathscr X\an$ vers $\mathscr X$ envoie une semi-norme sur son noyau. 
Fixons un système fini de générateurs $a_1,\ldots, a_m$ de $A$, et soit $I$
le noyau de la surjection $T_i\mapsto a_i$ de $k[T_1,\ldots,T_m]$ sur $A$. 
 Pour tout $R>0$ le sous-ensemble $V_R$ de $\mathscr X\an$ constitué des points
 $x$ tels que $\abs{a_i(x)}\leq R$ pour tout $i$ est un domaine affinoïde de 
 $\mathscr X^\an$, 
d'algèbre affinoïde associée $k\{T_1/R,\ldots, T_n/R\}/I$. L'espace $\mathscr X^\an$ 
est la réunion croissante des $V_R$, qui en constituent un G-recouvrement. 

On dispose d'un théorème GAGA dans ce contexte : si $\mathscr X$ est propre et si $\pi$ désigne le morphisme structural 
de $\mathscr X^\an$ vers $\mathscr X$ alors $\mathscr F\mapsto \pi^*\mathscr F$ établit une équivalence entre la catégorie
des faisceaux cohérents sur $\mathscr X$ et celle des faisceaux cohérents sur $\mathscr X^\an$, qui préserve les groupes 
de cohomologie (\cite{poineau2010a}, annexe A). 

Le foncteur d'analytification commute à la formation des produits fibrés. 

Pour tout espace analytique $X$
et tout entier $n$, 
l'application $f\mapsto (f^*T_1,\ldots, f^*T_n)$ met en bijection $\Hom(X,\mathbf A^{n,\an}_k)$
et $\mathscr O_X(X)^n$.

\begin{exem}\label{exemple-etar}
Soit $r=(r_1,\ldots, r_n)$ un polyrayon. 
L'application $\sum a_I T^I\mapsto \max \abs{a_I}\cdot r^I$ de $k[T_1,\ldots, T_n]$ vers $\R_+$ est une
norme multiplicative prolongeant la valeur absolue de $k$. C'est donc un point $\eta_r$ 
de $\A^{n,\an}_k$, situé au-dessus du point générique de $\A^n_k$. Le corps $\hr {\eta_r}$ est le complété
de $k(T_1,\ldots, T_n)$ pour la valeur absolue induite par $\eta_r$.

Lorsque $r$ est $k$-libre le corps $\hr {\eta_r}$ s'identifie
au corps $k$-affinoïde $k_r$ décrit en \ref{exemple-kr} (ce dernier contient 
en effet $k(T_1,\ldots, T_n)$ comme sous-corps dense), le singleton
$\{\eta_r\}$ est alors le domaine
affinoïde de $\A^{n,\an}_k$ défini par les égalités $\abs{T_i}=r_i$ pour tout $i$, et l'on a
$\mathscr O_{\A^{n,\an}_k}(\{\eta_r\})=k_r$. 
\end{exem}

\begin{exem}[La droite affine]
Pour tout $a\in k$ et tout réel $r$,
notons $\eta_{a,r}$
l'application $\sum a_i (T-a)î \mapsto \max \abs{a_i}r^i$
de $k[T]$ vers $\R_+$. C'est une semi-norme multiplicative, 
et donc un point de $\A^{1,\an}_k$ ; si $r=0$ c'est l'application 
$P\mapsto \abs{P(a)}$ de noyau $(T-a)$, et $\eta_{a,0}$ est donc 
le $k$-point standard d'équation $T=a$ ; si $r>0$ alors
$\eta_{a,r}$ est une norme, ce qui revient à dire que 
$\eta_{a,r}$ est situé au-dessus du point générique de
$\A^1_k$. On a 
$\eta_{a,r}=\eta_{b,s}$ si et seulement si $r=s\geq \abs{b-a}$. 
Si $\abs{k^\times}$ est dense dans $\R_+^\times$ 
on a $\abs{P(\eta_{a,r})}=\sup_{b\in k, \abs{b-a}\leq r}\abs{P(b)}$. 

On a $\hr{\eta_{a,0}}=k$. Supposons $r>0$ et posons
$\tau=\widetilde{(T-a)(\eta_{a,r})}$.Il résulte 
de l'exemple \ref{exem-ktr-tilde}
que $\tau$ est transcendant sur $\widetilde k$ et que $\hrt{\eta_{a,r}}=\widetilde k(\tau)$. 

Donnons quelques précisions supplémentaires dans deux cas extrêmes. 
Si $r=\abs \lambda$ avec $\lambda\in k^\times$ alors 
$\hrt{\eta_{a,r}}=\widetilde k(\tau/\widetilde \lambda)$. Comme
$\abs{\tau/\widetilde \lambda}=1$ on voit que $ \hrt {\eta_{a,r}}^1$
est égal à $\widetilde k^1(\tau/\widetilde \lambda)$, et est donc une extension rationnelle
de degré de transcendance $1$ de $\widetilde k^1$ ; 
de même, on a $\abs{\hr {\eta_{a,r}}^\times}=\abs{k^\times}$. 
Si $r\notin \abs{k^\times}^\Q$ alors $\hrt{\eta_{a,r}}^1=\widetilde k^1$ et 
 $\abs{\hr {\eta_{a,r}}^\times}=\abs{k^\times}\cdot r^\Z$.

Supposons maintenant que $k$ est algébriquement clos (auquel cas $\abs{k^\times}$
est divisible).
On
démontre alors (voir \cite{berkovich1990}, 1.4.4)
que tout point de $\A^{1,\an}_k$ est de l'un des quatre types suivants. 

\begin{itemize}
\item
Les points de type~I sont les points rigides,
c'est-à-dire ceux pour lesquels $\hr x=k$, ou encore
les points de la forme $\eta_{a,0}$ pour $a\in k$.

\item
Les points de type~II sont les points de la forme
$\eta_{a,r}$ où $a\in k$ et où $r\in \abs{k^\times}$. 
Si $x$ est de type~II le corps 
$\hrt x^1$ est une extension rationnelle de degré de transcendance~$1$
de $\widetilde k^1$ et $\abs{\hr x^\times}=\abs {k^\times}$. En particulier
$d_k(x)=1$.

\item
Les points de type~III sont les points de la forme
$\eta_{a,r}$ où $a\in k$ et où $r\notin \abs{k^\times}$. 
Si $x$ est de type III le corps 
$\hrt x^1$ est égale à 
$\widetilde k^1$ et $\abs{\hr x^\times}=\abs {k^\times}\cdot r^\Z$. En particulier
$d_k(x)=1$. 

\item
Les points de type~IV n'existent que si le corps~$k$ n'est pas maximalement
complet, c'est-à-dire s'il existe des familles totalement ordonnées de boules
fermées de $k$ d'intersection vide (c'est par exemple le cas si
$k=\C_p$, ou si $k$ est le complété du corps des séries de Puiseux sur $\C$). 
Si $(B(a_i,r_i))_i$ est une telle famille de boules, le point $x$ de type IV
qu'on lui associe est la semi-norme
$\inf_i \eta_{a_i,r_i}$, qui est en fait une norme.  Le point $x$
est donc situé au-dessus du point générique de $\A^{1,\an}_k$. Le corps complet 
$\hr x$ est une extension immédiate de $k$, c'est-à-dire que $\hrt x=\widetilde k$, ou encore que 
$\hr x$ a même groupe des valeurs et même corps résiduel (au sens classique) que $k$ ; et 
cette extension est stricte. On a $d_k(x)=0$. 
\end{itemize}

\end{exem}

\subsection{}\label{ss.paracompact}
Rappelons (\cite{bourbaki1971}, chap.~1, \S9, \no~10,
définition~6, p.~69)
qu'un espace topologique~$X$ est dit \emph{paracompact}
s'il est séparé et si tout recouvrement ouvert possède un raffinement
qui est localement fini. Un espace compact est paracompact.
Une partie fermée d'un espace paracompact est paracompacte.

Un espace localement
compact est dit dénombrable à l'infini s'il est réunion dénombrable
d'ensembles compacts; un espace localement compact
dénombrable à l'infini est paracompact.
Plus généralement,  un espace topologique localement compact
est paracompact si et seulement si c'est l'espace somme
d'une famille d'espaces localement compacts dénombrables à l'infini
(\loccit, théorème~5, p.~70).

\begin{lemm}[cf. \cite{berkovich1993}, remarque 1.2.4~(iii)]
\label{lemm.vois-paracompact}
Un espace analytique possède une base de voisinages
ouverts dénombrables à l'infini
(et donc paracompacts).
\end{lemm}
Cet énoncé figure dans~\cite[2.2.8 et 3.2.1]{berkovich1990} 
pour les espaces affinoïdes; nous détaillons ci-dessous les arguments
de la remarque~1.2.4, (iii), de~\cite{berkovich1993} qui l'en déduit.

\begin{proof}
Soit $x$ un point de $X$. 
Nous allons montrer que tout voisinage de $x$ dans $X$
contient un voisinage ouvert de $x$ qui est dénombrable à l'infini. 
On peut pour ce faire restreindre $X$ et donc le supposer topologiquement séparé. 
Le point $x$ possède alors une base de voisinages compacts dans $X$ qui sont 
des réunions finies de domaines affinoïdes contenant $x$.
On peut donc supposer que $X$ est compact et est une union finie $V_1\cup V_2\ldots \cup V_n$
de domaines affinoïdes contenant $x$ ; on se donne un voisinage
ouvert $U$ de $x$ dans $X$ et on veut montrer qu'il contient un voisinage 
ouvert de $x$ qui est dénombrable à l'infini. 
Raisonnons par récurrence sur~$n$.
Si $n=1$ alors $X$ est affinoïde et le résultat est connu dans ce cas
comme on l'a rappelé. Supposons $n>1$ et le résultat 
vrai pour $n-1$. 
Posons $W=V_1\cup \dots\cup V_{n-1}$.
Par récurrence, $x$ possède
un voisinage ouvert et dénombrable à l'infini
$\Omega$ dans $W\cap U$, dont l'intersection 
$\omega$ avec $V_n$ est un voisinage ouvert et 
dénombrable à l'infini de $x$ dans $W\cap V_n\cap U$. 
Soit $(K_i)_{i\geq 0}$ une suite de parties compactes de~$W\cap V_n$
telles que $\omega =\bigcup_{i\geq0} K_i$ et $K_i\subset\mathring{K_{i+1}}$
pour tout~$i$.
Soit $i\geq 0$.
Soit $y\in K_i$; il existe un voisinage~$\Omega_y$
de~$y$ dans~$V_n$ tel que $\Omega_y\cap W\cap V_n\subset \mathring K_{i+1}$;
par le cas affinoïde on peut supposer,
quitte à le restreindre, 
que $\Omega_y$ est un ouvert dénombrable à l'infini de $V_n\cap U$.
L'espace~$K_i$ étant compact, il est recouvert par un nombre
fini de ces ouverts~$\Omega_y$ de $V_n$ ; leur réunion est un voisinage~$\Omega_i$
de~$K_i$ dans $V_n\cap U$,  ouvert et dénombrable à l'infini, tel 
que $\Omega_i\cap W\cap V_n\subset \mathring{K}_{i+1}$.
La réunion $\Omega'$ des~$\Omega_i$ est un voisinage ouvert de~$x$ dans~$V_n\cap U$,
ouvert et dénombrable à l'infini, et $\Omega'\cap W\cap V_n=\omega$.
La réunion $\Omega\cup \Omega'$ est dénombrable à l'infini et contient $x$. 
On peut écrire $\Omega=\Upsilon\cap W$ et $\Omega'=\Upsilon'\cap V_n$
où $\Upsilon$ et $\Upsilon'$ sont deux ouverts de $X$ contenus dans $U$. 
Quitte à remplacer $\Upsilon$ par son intersection avec le complémentaire dans $X$ du
compact  $V_n\setminus \omega$, on peut supposer que $\Upsilon\cap V_n=\omega$ ; 
et quitte à remplacer $\Upsilon'$ par son intersection avec le complémentaire dans $X$ du
compact  $W\setminus \omega$, on peut supposer que $\Upsilon'\cap W=\omega$. 
Mais on a alors $\Omega\cup \Omega'=\Upsilon\cup \Upsilon'$, si bien que $\Omega\cup \Omega'$
est un voisinage ouvert de $x$ dans $U$. 
\end{proof}

\section{Topologie de Zariski, théorie de la dimension}

\subsection{Topologie de Zariski}

Soit $X$ un espace $k$-analytique. Pour tout faisceau cohérent d'idéaux $\mathscr I$ sur $X$
on note $V(\mathscr I)$
son lieu des zéros, c'est-à-dire l'ensemble des point $x$ de $X$ tel que $f(x)=0$ pour toute section $f$ de $\mathscr I$
définie sur un domaine analytique de $X$ contenant $x$. Les parties de $X$ de la forme $V(\mathscr I)$ sont les fermés
d'une topologie appelée \textit{topologie de Zariski} de $X$. 

Supposons que $X$ est affinoïde, et écrivons $X=\mathscr M(A)$. Se donner un faisceau cohérent d'idéaux
$\mathscr I$ sur
$X$ revient alors à se donner un idéal $I$ de $A$ (on a $\mathscr I=\widetilde I$
et $I=\mathscr I(X)$), et $V(\mathscr I)$ est alors l'image réciproque de $V(I)$ par la surjection continue
de $\mathscr M(A)$ vers $\Spec A$ (il est en particulier compact). On dispose donc dans ce cas d'une  bijection croissante entre l'ensemble des fermés
de Zariski de $\mathscr M(A)$ et l'ensemble des fermés de Zariski de $\Spec A$. En particulier la topologie de $X$
est noethérienne. 

On ne suppose plus $X$ affinoïde. Les fermés de Zariski de $X$ sont fermés dans $X$ (cela se vérifie G-localement, ce
qui perrmet de se ramener au cas affinoïde). Si $X$ est compact, sa topologie de Zariski est noethérienne (on le déduit du cas affinoïde
en écrivant $X$ comme une union finie de domaines affinoïde). 

\subsection{Sous-espaces analytiques fermés}

Soit $X$ un espace $k$-analytique et soit $\mathscr I$
un faisceau cohérent d'idéaux sur $X$. Supposons d'abord que $X$ est affinoïde, 
et écrivons $X=\mathscr M(A)$. Posons $I=\mathscr I(X)$.
Le fermé $V(\mathscr I)$ coïncide avec l'ensemble des $x$ dans $\mathscr 
M(A)$ tel que $f(x)=0$ pour tout $x\in I$, qui est homéomorphe à $\mathscr M(A/I)$ \textit{via}
le morphisme induit par la surjection canonique de $A$ vers $A/I$. Ainsi, $V(\mathscr I)$ peut être muni d'une structure 
d'espace $k$-analytique, dépendant du choix de l'idéal $\mathscr I$, et d'un morphisme de $V(\mathscr I)$
dans $X$ dont l'application continue sous-jacente est l'inclusion.

On ne suppose plus $X$ affinoïde. En procédant G-localement à la construction précédente puis en recollant
on obtient une structure d'espace $k$-analytique au-dessus de $X$ sur $V(\mathscr I)$, 
l'application continue sous-jacente au morphisme de $V(\mathscr I)$ dans $X$ étant
l'inclusion $i\colon V(\mathscr I)\hookrightarrow X$, et
$i_*\mathscr O_{V(\mathscr I)}=\mathscr O_X/\mathscr I$. 

Un \textit{sous-espace analytique fermé} de $X$ est un fermé de Zariski $Y$ de $X$ muni d'une structure d'espace
$k$-analytique au-dessus de $X$ 
induite selon le procédé décrit ci-dessus par un faisceau cohérent d'idéaux $\mathscr I$ tel que $Y=V(\mathscr I)$ ; 
notons que $\mathscr I$ se reconstitue alors à partir de cette structure comme le noyau de $\mathscr O_X\to i_*\mathscr O_Y$, où
$i\colon Y\to X$ est le morphisme structural, qui est topologiquement l'inclusion. 

Tout fermé de Zariski $Y$ possède au moins une structure de sous-espace analytique fermé : celle donnée 
par le faisceau
\[\mathscr J:=V\mapsto \{f\in \mathscr O(V), f(x)=0\;\text{pour tout}\;\in Y\cap V\,\]
qui est bien cohérent
(\cite{ducros2009}, Prop. 4.2.1). C'est la plus petite structure
de sous-espace analytique fermé sur $Y$, au sens où $\mathscr J$ est le plus grand faisceau cohérent d'idéaux sur $X$ définissant
le fermé $Y$ ; c'est aussi la seule structure de sous-espace analytique fermée sur $Y$ dont le faisceau structural n'a pas de sections
nilpotentes non nulles ; on l'appelle la structure réduite sur $Y$. Si $Y=X$ la structure réduite est associée au faisceau 
d'idéaux dont les sections sont les fonctions localement nilpotentes ; le sous-espace analytique fermé correspondant
est noté $X_\red$. Le morphisme $X_\red\hookrightarrow X$ induit un homéomorphisme entre $X_\red$ et $X$ pour leur topologies usuelles, 
mais aussi pour leurs topologies de Zariski. 

Si $Y$ est un sous-espace analytique fermé de $X$, il résulte du théorème de Gerritzen-Grauert dans sa version due à Temkin
(\cite{temkin2005}, Th. 3.1) que les domaines analytiques de $Y$ sont les sous-ensembles de $Y$ qui admettent un G-recouvrement 
par des parties de la forme $V\cap Y$ où $V$ est un domaine affinoïde de $X$.

Si $V$ est un domaine analytique de $X$ la topologie de Zariski de $V$ est plus fine que la topologie de $V$
induite par la topologie de Zariski de $X$, et en général strictement plus fine. 

Si $E$ est une partie de $X$ et si $(X_i)$ est un G-recouvrement de $X$ par des domaines analytiques, alors $E$
 est un fermé (resp. un ouvert) de Zariski de $X$ si et seulement si $E\cap X_i$ est un fermé (resp. un ouvert) de Zariski
 de $X_i$ pour tout $i$ (\cite{ducros2009}, Prop. 4.2.1). 

Une \textit{immersion fermée}\footnote{Dans l'expression « immersion fermée » ou
dans l'expression « sous-espace analytique fermé », 
l'adjectif « fermé » fait implicitement référence à la topologie de Zariski. 
Ainsi, en général, un domaine analytique fermé $V$ de $X$
n'est pas un sous-espace analytique fermé de $X$, 
et le morphisme canonique $V\hookrightarrow X$ n'est pas une immersion fermée.}
est un morphisme d'espaces $k$-analytiques $f\colon Y\to X$ 
se factorisant par un isomorphisme
entre $Y$ et un sous-espace analytique fermé de $X$. 
Une immersion fermée préserve les corps résiduels complétés. 

Un morphisme d'espaces $k$-analytiques $f\colon Y\to X$ est dit séparé si sa diagonale $\delta\colon 
Y\to Y\times_X Y$ est une immersion fermée, et $X$ est dit séparé si le morphisme structural de 
$X$ vers $\mathscr M(k)$ est séparé. La séparation dans ce sens entraîne la séparation topologique, mais
la réciproque est fausse : le recollement de deux copies du disque unité le long du domaine affinoïde décrit
par l'égalité $|T|=1$ est topologiquement séparé car le domaine de recollement est fermé, mais n'est pas séparé.

\subsection{Dimension}\label{def-dimension}
On dispose d'une théorie de la dimension en géométrie analytique, due à Berkovich, au sujet de laquelle on pourra
consulter le chapitre 2 de \cite{berkovich1990}, la fin du paragraphe 1.2 de \cite{berkovich1993} ou encore l'article
\cite{ducros2007}. 

Soit $X$ un espace $k$-analytique. Sa dimension ($k$-analytique) est définie comme suit. 
Commençons par le cas affinoïde. Écrivons $X=\mathscr M(A)$
et choisissons une extension complète $L$ de $k$ telle
que $A_L:=A\widehat{\otimes}_k L$ soit strictement $L$-affinoïde. La dimension de Krull de $\mathscr A_L$ ne dépend alors
pas de $L$ (\cite{ducros2007}, 1.5)
et est appelée dimension de $X$. 
(Si $X$ est strictement $k$-affinoïde sa dimension est donc simplement sa dimension de Krull relative à sa topologie
de Zariski.) On a par construction $\dim X_\red=\dim X$ et $\dim X_L=\dim X$ pour toute extension complète $L$ de $k$.

Si  $Y$ est un fermé de Zariski de $X$, le choix d'une structure de sous-espace analytique fermé sur $Y$, c'est-à-dire
d'un idéal $I$ de $A$ définissant $Y$, fait de $Y$ un espace affinoïde (identifié à $\mathscr M(A/I)$). 
Sa dimension ne dépend pas du choix de $I$, et nous parlerons donc de la dimension de $Y$ sans spécifier $I$. 

Soit $X_1,\ldots, X_n$
les fermés de Zariski de $X$ correspondant aux idéaux premiers minimaux de $A$. Si $V$ est un domaine affinoïde de $X$
on a alors $\dim V=\sup_{i, V\cap X_i\neq \emptyset}\dim X_i$. En particulier $\dim V\leq \dim X=\sup_i \dim X_i$. 
Si $x$ est un point de $X$ on définit $\dim_x X$ comme le minimum des $\dim V$ où $V$ est un domaine affinoïde de $X$
contenant $x$ ; c'est aussi $\sup_{i,x\in X_i}\dim X_i$. Il vient $\dim X=\sup_i \dim X_i=\sup_{x\in X} \dim_x X$. 

Ne supposons plus $X$ affinoïde. Si $x$ est un point de $X$, on définit la dimension locale $\dim_x X$ de $X$
en $x$ comme le minimum des $\dim V$ où $V$ parcourt l'ensemble des domaines affinoïdes de $X$ contenant $x$, et l'on pose alors
$\dim X=\sup_{x\in X}\dim_x X$. Cette définition est compatible avec la précédente lorsque $X$ est affinoïde. 
Si $V$ est un domaine analytique de $X$ on a $\dim_x V=\dim_x X$ pour tout $x\in V$, si bien que  $\dim V\leq \dim X$. 
Si $d$ est un entier on dit que $X$ est \textit{purement de dimension $d$} si $\dim_x X=d$ pour tout $x\in X$. Si c'est le cas 
$X$ est ou bien vide, ou bien de dimension $d$.

Pour tout point $x$ de $X$ on pose 
\begin{equation}\label{eq.dk(x)}
d_k(x)= \mathrm{degtr} (\hrt x/\tilde k)=\mathrm{degtr}(\hrt x^1/\widetilde k^1)+
   \dim_\Q\left(\Q\otimes_\Z( \abs{\hr x^\times}/\abs{k^\times})\right),
\end{equation}
où la seconde égalité provient de 
\ref{deg-extension-graduee} ; 
Berkovich le note $d(\kappa(x)/k)$ ou 
$d(\hr x/k)$
L'intérêt majeur de cet invariant est l'égalité~$\dim (X)=\sup_{x\in X}d_k(x)$ 
(\cite{berkovich1993}, Lemma 2.5.2).

Soit $Y$ un fermé de Zariski de $X$. Munissons-le d'une structure de sous-espace analytique fermé. La dimension de l'espace
analytique correspondant ne dépend que de $Y$, et pas de la structure en question (par
exemple au vu de l'égalité
ci-dessus, ou par réduction au cas affinoïde \textit{via} la description des domaines analytiques de $Y$).
On pourra donc parler de la dimension de $Y$ sans spécifier la structure dont on le munit. 

\begin{rema}
Insistons sur le fait que la dimension considérée ici est en fait une dimension relative au corps de base $k$, même si nous ne rappelons
pas cette dépendance. Considérons par exemple l'algèbre de Banach $k_r$ associée à un 
polyrayon $k$-libre $(r_1,\ldots, r_n)$ (exemple \ref{exemple-kr}). C'est une extension complète de $k$. En tant qu'algèbre
$k_r$-affinoïde elle est stricte, et sa dimension de Krull est nulle. Par conséquent la dimension
$k_r$-analytique de $\mathscr M(k_r)$ est nulle. 
Par ailleurs on a $k_r\widehat{\otimes}k_r=(k_r)_r\simeq k_r\{T, S\}/(ST-1)$, qui est 
strictement $k_r$-affinoïde de dimension de Krull $n$. La dimension $k$-analytique de 
$\mathscr M(k_r)$ est donc égale à $n$. 
\end{rema}

\subsection{Les composantes irréductibles}\label{ss-composantes-irreductibles}
Lorsque nous parlerons d'un espace $k$-analytique irréductible, ce sera toujours en référence à sa topologie de Zariski. 
Un espace $k$-analytique est noethérien pour la topologie de Zariski dès qu'il est compact, mais ce n'est pas le cas en général : 
par exemple si $k$ est non trivialement valué et si $f\in k[\![t]\!]$ est une série de rayon infini qui n'est pas un polynôme, 
le lieu des zéros de $f$ sur $\A^{1,\an}_k$ est un ensemble infini de points rigides tendant vers l'infini. C'est donc un fermé de Zariski
de $\A^{1,\an}$ possédant une infinité de composantes connexes. 

On peut néanmoins développer une théorie des composantes irréductibles en géométrie analytique sur $k$. Brian Conrad l'a fait  dans
\cite{conrad1999} (dans le cas strictement $k$-analytique) 
en procédant comme en géométrie complexe, c'est-à-dire en construisant d'abord le normalisé $\widetilde X$ d'un
espace $k$-analytique, puis en définissant les composantes irréductibles de $X$ comme les images des composantes connexes de 
$\widetilde X$. Dans ce texte, nous nous appuierons sur une approche différente présentée dans \cite{ducros2009}, que nous allons 
brièvement rappeler. 

On y démontre tout d'abord que si $X$ est un espace $k$-analytique irréductible, il est purement de dimension $d$
pour un certain entier $d$, et tout fermé de Zariski strict de $X$ est alors de dimension $<d$ (\cite{ducros2009}, cor. 4.14)
puis que si $X$ s'écrit comme une union localement finie $\bigcup_i X_i$ de fermés de Zariski irréductibles
alors $X$ est égal à l'un des $X_i$. 

Si $X=\mathscr M(A)$ est un espace affinoïde, sa topologie de Zariski est noethérienne et il possède donc un ensemble 
fini de composantes irréductibles pour cette topologie, que l'on appelle simplement les composantes 
irréductibles  de $X$. Ce sont les images réciproques des composantes irréductibles de $\Spec A$.

Si $X$ est un espace $k$-analytique quelconque, on appelle composante irréductible de $X$ toute partie de $X$
qui est l'adhérence pour la topologie de Zariski de $X$ d'une composante irréductible d'un domaine affinoïde de $X$. 
On démontre (\cite{ducros2009}, Th. 4.20) que les composantes irréductibles de $X$ constituent un ensemble localement fini de fermés de Zariski irréductibles
de $X$ qui recouvrent $X$, sont deux à deux non comparables pour l'inclusion et sont maximaux ; plus précisément, tout 
fermé irréductible de $X$ est contenu dans une composante irréductible de $X$. 
Si $E$ est un ensemble de composantes irréductibles de $X$,
la réunion $\bigcup_{Z\in E}Z$ est un fermé de Zariski de $X$ dont les 
éléments de $E$ sont les composantes irréductibles. 

Si $x$ est un point de $X$ l'entier $\dim_x X$ est le maximum des dimensions des composantes irréductibles de $X$ contenant 
$x$ (\cite{ducros2009}, Lemme 4.21). 

Soit $V$ un domaine analytique de $X$. Pour toute composante irréductible $Y$ de $X$, l'intersection $Y\cap V$ est une union (éventuellement
vide, éventuellement infinie) de composantes irréductibles de $V$ de même dimension que $Y$ ; pour toute composante irréductible $T$ de $V$, 
l'adhérence de Zariski $Z$ de $T$ dans $X$ est une composante irréductible de $X$ de dimension $\dim T$, et $T$ est une composante
irréductible de $Z\cap V$ (\cite{ducros2009}, prop. 4.22).  

Si $U$ est un ouvert de Zariski de $X$ les assertions suivantes sont équivalentes (\cite{ducros2009}, Lemme 4.24) 
\begin{itemize}
\item[(i)] $U$ est Zariski-dense ;
\item[(ii)] $U$ est dense pour la topologie usuelle de $X$ ; 
\item[(iii)] $U$ rencontre toutes les composantes irréductibles de $X$.
\end{itemize}
Et si elles sont satisfaites, $U\cap V$ est dense dans $V$ pour tout domaine analytique $V$ de $X$
(\cite{ducros2018}, Cor. 1.5.13). 

Soit $L$ une extension complète de $k$. Soit $(X_i)$ la famille des composantes irréductibles de $X$. 
Pour tout $i$ l'image réciproque $X_{i,L}$ de $X_i$ sur $X_L$ a une famille finie $(Y_{ij})_j$ de composantes irréductibles. 
Chacune des $Y_{ij}$ a même dimension que $X_i$ et se surjecte sur $X_i$ ; de plus $(Y_{ij})_{i,j}$ est la famille
des composantes irréductibles de $X_L$ (\cite{ducros2009}, Lemme 4.25 et théorème 7.1.16
(i), (ii) et (iii)).

Mentionnons qu'on dispose d'un principe GAGA pour les composantes irréductibles : si $\mathscr X$ est un schéma localement de type
fini sur une algèbre affinoïde $A$, et si les $\mathscr X_i$ sont les composantes irréductibles de $\mathscr X$, les $\mathscr X_i^\an$ sont les
composantes irréductibles de $\mathscr X^\an$ (\cite{ducros2018}, Prop. 2.7.16). Si $\mathscr X$ est irréductible et si 
$\mathscr Y$est un fermé irréductible de $\mathscr X$ on a alors 
\[\dim \mathscr X^\an=\dim \mathscr Y^\an+\text{\upshape{codim}}(\mathscr Y,\mathscr X)\]
(\cite{ducros2018}, Lemma 2.7.10 ; la preuve est un peu plus simple lorsque $\mathscr X=\Spec A$, 
voir \cite{ducros2007}, Prop. 1.11). 

\begin{lemm}\label{lemm-adherence-point}
Soit $X$ un espace $k$-analytique et soit $x\in X$. Les assertions suivantes sont équivalentes : 

\begin{itemize}
\item[\text{(i)}] l'adhérence de Zariski $\overline{\{x\}}^{X_{\mathrm{Zar}}}$ est une composante irréductible de $X$ ; 
\item [\text{(ii)}] la dimension de $\overline{\{x\}}^{X_{\mathrm{Zar}}}$ est égale à $\dim_x X$. 
\end{itemize}
Si elles sont satisfaites,  $\overline{\{x\}}^{X_{\mathrm{Zar}}}$ est l'unique composante irréductible de $X$ contenant $x$.
\end{lemm}

\begin{proof}
Choisissons une composante irréductible $Z$ de $X$ contenant $x$ et de dimension $\dim_x X$. 
Le fermé de Zariski $\overline{\{x\}}^{X_{\mathrm{Zar}}}$  est alors contenu dans la composante 
irréductible $Z$, et les assertions (i) et (ii) sont donc toutes deux 
équivalentes à l'égalité $\overline{\{x\}}^{X_{\mathrm{Zar}}}=Z$. 
Si elles sont satisfaites et si $T$ est une composante irréductible de $X$ distincte de $Z$ alors $T\cap Z$
est un fermé de Zariski strict de $Z=\overline{\{x\}}^{X_{\mathrm{Zar}}}$, si bien que $T$ ne peut contenir $x$. 
\end{proof}

\section{Propriétés locales des espaces analytiques et des faisceaux cohérents}

\subsection{}\label{ss-liste-proprietes}
Soit $\mathsf P$ une des propriétés 
suivantes des anneaux locaux noethériens : 
vérifier la condition~$\mathrm R_m$, 
être régulier, être Gorenstein, 
être intersection complète. Soit $\mathsf Q$ une des propriétés suivantes 
d'un module de type fini sur un anneau local noethérien: être libre de rang~$r$, 
vérifier la condition~$\mathrm S_n$, 
être de Cohen-Macaulay

Soit $X$ un espace $k$-analytique et
soit $x$ un point de $X$. Soit $\mathscr F$ un faisceau cohérent sur $X$. 
Pour tout bon domaine analytique $V$ de $X$
contenant $x$, nous noterons $\mathscr F_{V,x}$ la fibre en $x$ de
$\mathscr F|_V$ vu comme faisceau sur le site topologique de $V$, c'est-à-dire
encore la colimite des $\mathscr F(U)$
où $U$ parcourt l'ensemble des voisinage ouverts de $x$ dans $V$. 
S'il existe un bon domaine analytique ~$V$ de~$X$ contenant~$x$
tel que 
l'anneau local noethérien $\mathscr O_{V,x}$
(resp. tel que le $\mathscr O_{V,x}$-module 
de type fini $\mathscr F_{V,x}$)
satisfasse $\mathsf P$ (resp. $\mathsf  Q$) 
ceci reste vrai pour tout bon domaine analytique de~$X$ qui contient~$x$
(théorème~3.4 de~\cite{ducros2009}).
On dit alors que $X$ (resp. $\mathscr F$)
satisfait la propriété~ $\mathsf P$ (resp. $\mathsf Q$)
en~$x$.

Le lieu des point en lesquels la propriété $X$
(resp. $\mathscr F$) satisfait $\mathsf P$ (resp. $\mathsf Q$) 
est un ouvert de Zariski de~$X$
(\cite{ducros2009}, théorème~4.4). 

On dit que $X$ (resp. $\mathscr F$) satisfait la propriété~$\mathsf P$ (resp. $\mathsf Q$)
s'il la satisfait
en chacun de ses points. Si $U$ est un domaine analytique de $X$ contenant $x$
alors $X$ (resp. $\mathscr F$) satisfait $\mathsf P$ (resp. $\mathsf Q$) en $x$ si et seulement si c'est le cas 
de $U$ (resp. $\mathscr F|_U$). 

Mentionnons une entorse à ces principes généraux : si $\mathscr F$ est libre de rang $r$ en tout point de $X$
nous ne dirons évidemment pas que $\mathscr F$ est libre de rang $r$, mais qu'il est localement libre de rang $r$ --
et d'après un résultat de Berkovich (\cite{berkovich1993}, 1.3.4) déjà rappelé et discuté en \ref{definition-gtopologie}, 
cela équivaut à dire que $\mathscr F$ est localement libre en tant que $\mathscr O_X$-module
sur le site annelé $X_\groth$.

Lorsque $\mathscr F=\mathscr O_X$ nous dirons souvent
«$X$ satisfait $\mathsf Q$ en $x$» plutôt que «$\mathscr O_X$ satisfait $\mathsf Q$ en $x$». 

Nous étendrons
les conventions ci-dessus à la conjonction de différentes propriétés d'algèbre commutative ; nous 
parlerons ainsi d'espace réduit (c'est-à-dire satisfaisant 
$(\mathrm R_0)$ et $(\mathrm S_1)$) ou normal (c'est-à-dire satisfaisant 
$(\mathrm R_1)$ et $(\mathrm S_2)$) en un point $x$, ou globalement. 
Tout espace $X$ possède un normalisé $\widetilde X$, muni d'un morphisme fini 
surjectif $\widetilde X\to X$ qui induit une bijection entre $\pi_0(\widetilde X)$ et l'ensemble
des composantes irréductibles de $X$ (\cite{ducros2009}, Déf. 5.10, Thm. 5.13, Thm. 5.17) ; 
la normalisation commute à l'analytification (\cite{ducros2018}, Lemma 2.7.15).

Soit $L$ une extension complète de $k$ et soit $\pi$ le morphisme 
structural de $X_L$ vers $X$. 
Le faisceau cohérent $\mathscr F$ satisfait $\mathsf Q$ en $x$ si et seulement si $\pi^*\mathscr F$ satisfait
$\mathsf Q$ en $y$ ; si
si $\mathsf P$ est la propriété d'être Gorenstein ou d'intersection complète, $X$ satisfait $\mathsf P$ en $x$ 
si et seulement si $X_L$ la satisfait en $y$ ; si $\mathsf P$ est la propriété d'être régulier ou de vérifier $(\mathrm R_n)$
alors $X$ satisfait $\mathsf P$ en $x$ dès que $X_L$ satisfait $\mathsf P$ en $y$, et la réciproque est vraie si $L$
est analytiquement séparable sur $k$ (\cite{ducros2018}, Prop. 2.6.7), mais elle ne l'est pas en général. 

Si $\mathsf P$ est la propriété d'être régulier ou de vérifier $(\mathrm R_m)$, ou une propriété impliquant d'une façon ou d'une autre la régularité
(comme le fait d'être réduit, ou normal), on est dès lors conduit à introduire la notion de validité \textit{géométrique}
de $\mathsf P$: 
on dit que $X$ satisfait la propriété $\mathsf P$ \textit{géométriquement} en $x$ si pour toute extension complète $L$ 
de $k$ et tout antécédent $y$ de $x$ sur $X_L$, l'espace $X_L$ satisfait $\mathsf P$ en $y$. Pour que ce soit le cas, 
il suffit qu'il existe une extension parfaite $L$ de $k$ et un antécédent $y$ de $x$ sur $X_L$ tel que $X_L$
satisfasse $\mathsf P$ en $y$ (\cite{ducros2018}, 2.6.9). 

Soit $\mathscr X$ un schéma localement de type fini sur une algèbre affinoïde $A$ et soit $\mathscr G$ 
un faisceau cohérent sur $\mathscr X$. Soit $x$ un point de $\mathscr X^\an$ d'image $\xi$ sur $\mathscr X$. 
Le schéma $\mathscr X$ (resp. le faisceau $\mathscr G$) satisfait $\mathsf P$ (resp. $\mathsf Q$) en $\xi$ si et seulement si
$\mathscr X^\an$ (resp. $\mathscr G^\an$) la satisfait en $x$ (\cite{ducros2018}, Lemma 2.4.6). En étendant les scalaires au complété
d'une clôture algébrique de $k$, on obtient un principe GAGA analogue pour les propriétés géométriques.

\section{Propriétés des morphismes}

\subsection{Bord et intérieur}
Soit $f \colon Y\ra X$ un morphisme d'espaces $k$-analytiques.
On note $\Int(Y/X)$ l'intérieur relatif de~$f$
et $\partial(Y/X)$ son complémentaire, le bord relatif de~$f$
(\cite{berkovich1993}, Definition~1.5.4). Par construction,
$\Int(Y/X)$ est un ouvert de~$Y$.
Si $\partial(Y/X)=0$, nous dirons que $f$ est \emph{sans bord}
(la terminologie de~\cite{berkovich1993} est \emph{closed}).
Lorsque $f$ est l'inclusion d'un domaine analytique,
$\Int(Y/X)$ coïncide avec l'intérieur~$\mathring Y$
de~$Y$ dans~$X$ au sens de la topologie générale
et $\partial(Y/X)$ coïncide donc avec le bord topologique de~$Y$
dans $X$
(\cite{berkovich1993}, Prop.~1.5.4).

Si $\mathscr X$ est un schéma localement de type fini sur une algèbre
affinoïde $A$, le morphisme structural de $\mathscr X^\an$ vers
$\mathscr M(A)$ est sans bord. 

Si $g\colon Z\ra Y$ et $f\colon Y\ra X$ sont deux morphismes
d'espaces $k$-analytiques alors 
(\cite{berkovich1993}, Prop. 1.5.5)
\begin{equation}\label{incl.interieurs}
\Int(Z/Y) \cap g^{-1}(Y/X) \subseteq \Int(Z/X),
\end{equation}
et l'on a 
\begin{equation}\label{eq.interieurs}
\Int(Z/Y) \cap g^{-1}(Y/X) = \Int(Z/X),
\end{equation}
si $f$ est localement séparé
(\cite{temkin2004}, Corollary~5.7), ce qui est notamment le cas dès que $Y$ est
séparé. 

Lorsque $X=\mathscr M(k)$, on note $\Int(Y)=\Int(Y/k)$ et $\partial(Y)=\partial(Y/k)$,
qu'on appelle l'intérieur et le bord de~$Y$ respectivement. Nous dirons que 
$Y$est sans bord si $\partial(Y)=\emptyset$. Si $Y$ est sans bord, ou plus généralement si 
$f\colon Y\to X$ est sans bord et si $X$ est 
bon, l'espace $Y$ est bon. 
Si $f\colon Y\ra X$ est un morphisme d'espaces $k$-analytiques
et que $X$ est séparé, le cas \ref{eq.interieurs} entraîne les égalités
\begin{equation}\label{eq.interieurs2}
\Int(Y/X) \cap f^{-1}(\Int(X))=\Int(Y)
\quad\text{et}\quad
\partial(Y) = \partial(Y/X) \cup f^{-1}(\partial(X)). 
\end{equation}

Si $Y\to X$ est un morphisme d'espaces $k$-analytiques et si si $X'$ est un espace analytique sur une extension complète de $k$
muni d'un $k$-morphisme vers $X$, on a l'inclusion $\Int(Y/X)\times_X X'\subset \Int(Y\times_X X'/X')$
(cela découle de la proposition 3.13 de \cite{berkovich1990}). 
La classes des morphismes sans bord est en particulier stable par changement de base analytique sur $k$. 
Elle est stable par composition d'après \ref{eq.interieurs} et est G-locale au but
(\cite{temkin2004}, Cor. 5.6).

\begin{exem}
1)
Soit $r$ un nombre réel~$s>0$. Soit $D$ le domaine affinoïde de $\A^{1,\an}_k$ 
défini par l'inégalité $\abs T\leq r$ ; c'est le disque fermé de rayon $r$
au sens de Berkovich, isomorphe à $\mathscr M(k\{T/r\})$. 
Son bord topologique dans $\A^{1,\an}_k$ est le singleton
$\{\eta_r\}$ (défini à l'exemple \ref{exemple-etar}). 
On a donc 
aussi $\{\eta_r\}=\partial (D/\A^{1,\an}_k)=\partial (D)$, la seconde égalité provenant
du fait que $\A^{1,\an}_k$ est sans bord. 

Si $E$ désigne le domaine analytique fermé (mais non compact) de $\A^{1,\an}_k$ défini par l'inégalité
$\abs T\geq r$, le bord topologique de $E$ dans $\A^{1,\an}_k$ est encore
$\{\eta_r\}$, si bien que $\{\eta_r\}=\partial (E/\A^{1,\an}_k)=\partial (E)$.

2) Soit $X$ un espace $k$-analytique, soit 
$f\in\mathscr O(X)$
et soit $W$ le domaine analytique de~$X$
défini par l'inégalité $\abs f\leq r$.
La fonction $f$  induit un   morphisme de $X$ vers $\A^{1,\an}_k$, encore noté $f$ ; 
avec cette convention, on a $W=f^{-1}(D)$. 
Comme $\A^{1,\an}$ est sans bord, $\partial(W)=\partial(W/\A^{1,\an})$. La
formule~\eqref{eq.interieurs2} appliquée
au diagramme $W\ra D\ra\mathbf A^{1,\an}$, 
couplée au fait que $\partial(D/\A^{1,\an}_k)=\{\eta_r\}$, 
assure que 
$\partial(W)=f^{-1}(\eta_r) \cup (\partial(W/D))$.
Et
la même formule appliquée à $W\to X\to \mathscr M(k)$
entraîne que $\partial(W)=\partial(W/X)\cup (\partial(X)\cap W)$.
Cela démontre que $\partial(W)$ 
contient $(\partial(X)\cap W)\cup f^{-1}(\eta_r)$.
Par ailleurs, la propriété pour un morphisme  d'être sans
bord étant stable par changement de base, le morphisme
$W\setminus f^{-1}(\eta_r)\ra X$  est sans bord; autrement
dit, $\partial( W/X)\subset f^{-1}(\eta_r)$.
On en déduit 
\begin{equation}\label{lemm.bord(W-)}
  \partial(W) = (\partial(X)\cap W) \cup f^{-1}(\eta_r). 
\end{equation}
Si $W'$ est le domaine analytique de~$X$ défini par l'inégalité
$\abs f\geq r$, on démontre de même que
\begin{equation}\label{lemm.bord(W+)}
  \partial(W') = (\partial(X)\cap W') \cup f^{-1}(\eta_r). 
\end{equation}
\end{exem}

\subsection{}
Soit $f\colon Y\to X$ un morphisme d'espaces $k$-analytiques.
On dit que $f$ est \emph{compact} 
si l'image réciproque de toute partie compacte est compacte.
Comme tout point possède un voisinage compact, cela revient à dire
que $f$ est topologiquement propre et topologiquement séparé.

Suivant~\cite{berkovich1993} (1.5.3), un morphisme est dit \emph{propre}
s'il est compact et sans bord.

\subsection{Morphismes finis}
Soit $A$ une algèbre $k$-affinoïde. Une $A$-algèbre de Banach finie est une $A$-algèbre de Banach $B$ qui est finie comme
$A$-algèbre, et dont la norme est équivalente à la norme quotient induite par une présentation finie $A^n\to B$ du $A$-module $B$.
Dans ce cas la norme de $B$ est en fait équivalente à la norme quotient associée à n'importe quelle présentation finie 
du $A$-module $B$ (c'est une conséquence de \cite{berkovich1990}, Prop. 2.1.9). Toute $A$-algèbre de Banach finie est $k$-affinoïde
(combiner \cite{bosch-g-r1984}, Prop. 6.1.3/4  et \cite{berkovich1990}, Cor. 2.1.8). Le foncteur d'oubli de la norme
induit une équivalence de catégories entre $A$-algèbres de Banach finies et $A$-algèbres finies\footnote{Si $\abs{k^\times}\neq \{1\}$ 
et si $A$ est strictement $k$-affinoïde, on déduit de \cite{berkovich1990}, Prop. 2.1.1 
et \cite{bosch-g-r1984},
Thm. 6.1.3/1 que toute $A$-algèbre de Banach qui est 
finie comme $A$-algèbre est une $A$-algèbre de Banach finie. Par contre si $k$ 
est trivialement valué et si
$ 0<r<s<1$, le morphisme continu naturel de $k\{T/s\}$ vers $k\{T/r\}$ est bijectif (sa source
et son but s'identifient
tous deux à $k[[T]]$) mais sa réciproque n'est pas bornée, si bien que $k\{T/s\}$ n'est pas une $k\{T/r\}$-algèbre
de Banach finie).} (\cite{berkovich1990}, 
Prop. 2.1.12).

Soit $f\colon Y\to X$ un morphisme entre espaces $k$-analytiques. 
Si $X$ et $Y$ sont affinoïdes, on dit que $f$ est fini si $f^*$ fait de $\mathscr O(Y)$
une $\mathscr O(X)$-algèbre de Banach finie ; cela revient à demander que $f$ soit sans bord
(\cite{berkovich1990}, Cor. 2.5.13).
En général (si l'on ne suppose plus $X$ et $Y$ affinoïdes) les assertions suivantes sont équivalentes : 
(\cite{berkovich1993}, Cor. 1.3.7). 
\begin{itemize}
\item[(i)] Il existe un G-recouvrement $(X_i)$ de $X$ par des domaines
affinoïdes tels que pour tout $i$, l'image réciproque $f^{-1}(X_i)$ soit affinoïde et que le morphisme induit de $f^{-1}(X_i)$ vers $X_i$ soit fini
\item[(ii)] Pour tout domaine affinoïde $V$ de $X$ l'image réciproque $f^{-1}(V)$ est affinoïde et le morphisme induit
de $f^{-1}(V)$ vers $V$ est fini. 
\end{itemize}
On dit que $f$ est fini si elles sont satisfaites ; cette définition est compatible avec la précédente si $X$ et $Y$ sont affinoïdes. 
Un morphisme fini est sans bord. 

La classe des morphismes finis est stable par composition et changement de base analytique sur $k$, 
et la propriété d'être fini est G-locale au but. 

Soit $y$ un point de $Y$. On dit que $f$ est fini en $y$ s'il existe un voisinage analytique $U$ de $f(y)$ dans $X$ et un voisinage 
analytique $V$ de $y$ dans $f^{-1}(U)$ tel que $f$ induise un morphisme fini de $V$ vers $U$. 
Si c'est le cas le point $y$ appartient à $\Int(Y/X)$, et si de plus $X$ et $Y$ sont bons $\mathscr O_{Y,y}$ est 
un $\mathscr O_{X,x}$-module de type fini. 
La finitude en un point est stable par composition
et changement de base analytique sur $k$.

\subsection{Dimension relative}
Soit $f\colon Y\to X$ un morphisme d'espaces $k$-analytiques et soit $y$ un point de $Y$ ; posons
$x=f(y)$. 
On appelle dimension (relative) de $f$ en $y$, ou encore dimension relative de $Y$ sur $X$ en $y$, 
et l'on note $\dim_y f$, la dimension de l'espace $\hr x$-analytique $f^{-1} (x)$ en $y$. Nous dirons que $f$ est
\textit{quasi-fini} en $y$ (resp. que $f$ est quasi-fini) si $\dim_yf=0$ (resp. si $f$ est partout de
dimension relative nulle). 

L'application $y\mapsto \dim_y f$ est semi-continue supérieurement pour la topologie de Zariski
(\cite{ducros2007}, Théorème 4.9). 

Le morphisme $f$ est fini en un point $y$ de $Y$ si et seulement si $f$ est quasi-fini en $y$ et
$y\in \Int(Y/X)$ (\cite{berkovich1993}, Cor. 3.1.10). 

\subsection{Morphismes plats}
Soit $f\colon Y\to X$ un morphisme d'espaces $k$-analytiques. 
Soit $y$ un point de $Y$ et soit $x$ son image sur $X$.

Supposons tout d'abord que $X$ et $Y$ sont bons. On dit 
que $f$ est plat en $y$, ou que $Y$ est plat sur $X$ en $y$, 
si pour tout morphisme $X'\to X$ dont la source $X'$ est un bon espace analytique sur $k$, et tout point $y'$ de $Y\times_X X'$ situé
au-dessus de $y$ et dont on note $x'$ l'image sur $X'$, le $\mathscr O_{X',x'}$-module $\mathscr O_{Y\times_XX',y'}$
est plat.

Ne supposons plus que $X$ et $Y$ sont bons. On dit alors que 
$f$ est plat en $y$ s'il existe un bon domaine 
analytique $U$ de $X$ contenant $x$ et un bon domaine analytique $V$ de $f^{-1}(U)$ contenant $y$ tels que 
$f|_V\colon V\to U$ soit plat en $y$ ; et c'est alors le cas pour tout tel couple $(U,V)$ (\cite{ducros2018}, 4.1.7). 

Soit maintenant $U$ un domaine analytique quelconque (non nécessairement bon) de $X$ contenant $x$, et soit $V$ un domaine
analytique quelconque de $f^{-1}(U)$ contenant $y$. Il résulte aussitôt de la définition que $f$
est plat en $y$ si et seulement si
$f|_V\colon V\to U$ est plat en $y$. 

On dit que $f$ est plat s'il est plat en tout point de $Y$. 
La platitude est préservée par composition et changement de base (la définition est conçue pour garantir cette dernière
propriété). 

Le morphisme structural d'un espace $k$-analytique vers $\mathscr M(k)$ est plat (\cite{ducros2018}, Lemma 4.1.13). 

L'ensemble des points de $Y$ en lesquels $f$ est plat est un ouvert de Zariski de $Y$ (\cite{ducros2018}, Thm. 10.3.2). 

Soit $\phi\colon \mathscr Y\to \mathscr X$ un morphisme entre schémas de type fini sur une algèbre affinoïde
et soit $y$ un point de $\mathscr Y^\an$ d'image $\eta$ sur $\mathscr Y$. Le morphisme $\phi^\an$ est plat en $y$ si et seulement 
si $\phi$ est plat en $\eta$ (\cite{ducros2018}, Lemma 4.2.1 et Prop. 4.2.4).

\begin{rema}
Supposons que $X$ et $Y$ sont bons. 
Si $f$ est plat en $y$, l'anneau local $\mathscr O_{Y,y}$ est plat sur $\mathscr O_{X,x}$. 
La réciproque est fausse en général, le problème étant que
la platitude de $\mathscr O_{Y,y}$ est plat sur $\mathscr O_{X,x}$ n'est pas en général
préservée par changement de base bon (un contre-exemple dû à Temkin est étudié en détail au paragraphe 4.4 de \cite{ducros2018}). 

La platitude de $\mathscr O_{Y,y}$ sur $\mathscr O_{X,x}$
entraîne toutefois la platitude de $f$ en $y$ lorsque $y$ est un point rigide (\textit{cf.}
\cite{ducros2018}, 4.2.3) ou plus généralement lorsque $y$ appartient à $\Int(Y/X)$ (\cite{ducros2018}, 
Thm. 8.3.4). 
\end{rema}

\subsection{Morphismes quasi-lisses}

Soit $f\colon Y\to X$ un morphisme entre espaces $k$-analytiques. On définit comme en géométrie algébrique 
le $\mathscr O_Y$-module des différentielles de Kähler relatives $\Omega_{Y/X}$ (\cite{berkovich1990}, \S 3.3) ; c'est un 
faisceau cohérent sur $Y$. On dispose en particulier sur tout espace $k$-analytique $X$ du faisceau cohérent $\Omega_{X/k}$. 
Si $x$ est un point de $X$, la tige $\Omega_{X/k}\otimes \hr x$ est de dimension $\geq \dim_x X$ 
(\cite{ducros2009}, Lemma 6.2 (i)) et l'on dit que $X$
est quasi-lisse en $x$ si cette dimension vaut exactement $\dim_x X$. L'espace $X$
est quasi-lisse en $x$ si et seulement s'il est géométriquement régulier en $x$
(\cite{ducros2009}, Prop. 6.3) et si $x$ est un $k$-point, il suffit pour que ce soit
le cas que $X$ soit régulier en $x$ (\cite{ducros2009}, Lemma 6.2 (iii); l'hypothèse que $\abs{k^\times} \neq \{1\}$ figurant
dans son énoncé  est inutile, et n'est pas utilisée dans la preuve).

Cette notion admet une variante relative : dans  \cite{ducros2018} (déf. 5.2.4) on définit la notion de morphisme quasi-lisse en un point
au moyen d'un critère jacobien 
(et un morphisme est dit quasi-lisse s'il l'est en tout point de sa source). 
On démontre (\cite{ducros2018}, Thm. 5.3.4)
que $f\colon Y\to X$ est quasi-lisse en $y$ si et seulement si $f$ est plat en $y$ et la fibre $f^{-1}(f(y))$ est 
quasi-lisse  en $y$. 
Si $d$ est un entier le lieu $U_d$ des points de $Y$ en lesquels $f$ est quasi-lisse et de dimension relative $d$ est un ouvert de Zariski de $Y$
(\cite{ducros2018}, Thm. 10.7.2 (2))  et $\Omega_{U_d/X}$ est G-localement libre de rang $d$ (\cite{ducros2018}, Cor. 5.3.2 ; la preuve dans le
cas absolu où $X=\mathscr M(k)$ est beaucoup plus simple, \textit{cf.} \cite{ducros2009}, Prop. 6.6). 

On dit que $f$ est quasi-étale en $y$ s'il est quasi-lisse et quasi-fini en $y$, et que $f$ est quasi-étale s'il est quasi-étale en tout point de $Y$.

\begin{rema}
Berkovich a introduit au paragraphe 3 de
\cite{berkovich1994} une notion de morphisme quasi-étale ; même si sa définition diffère en apparence de celle que nous donnons ici, 
les deux sont en fait équivalentes (\cite{ducros2018}, Lemma 5.4.11). 

Berkovich a par ailleurs défini aux paragraphes 3.3 et 3.5 
de  \cite{berkovich1993} des notions de morphisme étale et lisse. Si $Y$ et $X$ sont bons, le morphisme $f$ est lisse en $y$
si et seulement s'il est quasi-lisse et intérieur en $y$ (\cite{ducros2018}, Cor. 5.4.8). Il est vraisemblable que l'hypothèse de bonté 
est nécessaire pour que cette équivalence soit valable ; on peut toutefois s'en passer en dimension relative nulle : même si l'on ne
suppose plus $X$ et $Y$ bons, $f$ est  étale en $y$ si et seulement si $f$ est quasi-étale et intérieur en $y$ (\cite{ducros2018}, 
Remark 5.4.9). 
\end{rema}

\begin{lemm}\label{lemm.voisinage}
Soit $X$ un espace $k$-analytique, soit $V$ un  domaine analytique
et soit $Y$ un fermé de Zariski de~$X$ qui est contenu dans~$V$.
Alors, $V$ est un voisinage de~$Y$ dans~$X$.
\end{lemm}
\begin{proof}
Commençons par traiter le cas où $V$ est séparé.
L'injection de~$Y$ dans~$X$ est sans bord; comme $V$ est supposé séparé,
$Y$ est contenu dans $\Int(V/X)$ (formule~\eqref{eq.interieurs2}),
lequel est égal à l'intérieur topologique de~$V$ dans~$X$.

Traitons maintenant le cas général. Soit $y$ un point de~$Y$
et soit $(X_1,\dots,X_n)$ une famille finie 
de domaines affinoïdes de~$X$ contenant~$y$ 
et dont la réunion est un voisinage de~$y$ dans~$X$.
Pour tout~$i$, $X_i\cap V$ est séparé, de sorte qu'il existe
un ouvert~$U_i$ de~$X$ tel que $U_i\cap Y$ soit contenu dans~$X_i\cap V$.
Soit $U$ l'intersection de~$U_1\cap \dots\cap U_n$
et de l'intérieur de $X_1\cup\dots\cup X_n$;
c'est un ouvert de~$X$ qui contient~$y$.
Il est aussi contenu dans la réunion des~$U_i\cap X_i$, donc dans~$V$.
\end{proof}

\subsection{}
Soit $f\colon Y\to X$ un morphisme d'espaces $k$-analytiques.
On dit que $f$ est \emph{compact} 
si l'image réciproque de toute partie compacte est compacte.
Comme tout point possède un voisinage compact, cela revient à dire
que $f$ est topologiquement propre et topologiquement séparée.

Suivant~\cite{berkovich1993} (1.5.3), un morphisme est dit \emph{propre}
s'il est compact et sans bord.

\subsection{}
Rappelons que les propriétés pour~$f$ d'être \emph{fini} 
(resp. \emph{fini et plat}, resp.~\emph{compact}, resp.~\emph{propre})
sont locales sur le but 
(voir \cite{berkovich1993}, proposition~3.1.8 et corollary~3.2.3;
\cite{temkin2004}, corollary~5.6).

\section{Utilisation de l'invariant $d_k(x)$}

Nous nous proposons ici d'établir un certain nombre de résultats techniques en géométrie analytique 
qui nous seront utiles par la suite et mettent en jeu de façon essentielle l'invariant $d_k(x)$ introduit
en \ref{eq.dk(x)}.

\begin{lemm}\label{lemm.fibre-dim-zero}
Soit $f\colon Y\to X$ un morphisme d'espaces analytiques.
Pour tout $x\in X$ on a
\[ \dim Y \geq \dim  Y_x+d_k(x). \]
En particulier si $Y_x$ est non vide et si $\dim Y\leq d_k(x)$ on a
$\dim Y=d_k(x)$ et $\dim Y_x=0$.
\end{lemm}
\begin{proof}
On a pour tout $y\in Y_x$ 
la minoration
\[\dim Y\geq d_k(y)=d_{\mathscr H(x)}(y)+d_k(x).\]
Il vient 
\[\dim Y\geq \sup_{y\in Y_x}d_{\hr x}(y)+d_k(x)=\dim Y_x+d_k(x).\]
\end{proof}

\begin{lemm}\label{lemm.hauptidealsatz-anal}
Soit $Y$ un espace $k$-analytique purement de dimension~$n$,
soit $d$ un entier et soit $f\colon Y\to\A^d_k$ un morphisme.
Soit $x$ un point de~$\A^d_k$. 
Pour tout  point~$y$  de~$Y_x$,
on a $\dim_y Y_x\geq n-d$. Si de plus $d_k(x)=d$, on a égalité.
\end{lemm}\begin{proof}
Pour démontrer que $\dim_y Y_x\geq n-d$ on peut 
remplacer $Y$ par n'importe quel domaine analytique 
de $Y$ contenant $y$. Et si $L$ est une extension complète de
$k$ et $y'$ un antécédent de $y$ sur $Y_L$ dont non note $x'$
l'image sur $X_L$, il suffit de démontrer que 
$\dim_{y'}Y'_{x'}\geq n-d$. 
On peut donc supposer que $Y$ est affinoïde
et que $x\in\A^d_k(k)$; en composant $f$
avec une 
translation, on peut même supposer que $y=0$. 
Notons $f_1,\ldots,f_d$
les composantes du morphisme $f$, 
et posons $\mathscr B=\mathscr O_Y(Y)$. 
La fibre $Y_x$ est l'espace affinoïde
$\mathscr M(\mathscr B/(f_1,\dots,f_d))$. 
D'après le Hauptidealsatz de Krull, les composantes
irréductibles de $\Spec(\mathscr B/(f_1,\dots,f_d))$
sont de codimension au plus~$d$ dans~$\Spec(\mathscr B)$.
Par conséquent, les composantes irréductibles
de~$Y_x$ sont de dimension au moins $n-d$, 
si bien que $\dim_y Y_x\geq n-d$. 

Supposons que $d_k(x)=d$. 
On a alors
\[\dim_y Y_x\leq \dim Y_x\leq \dim Y-d=n-d,\]
où
la seconde inégalité provient du lemme \ref{lemm.fibre-dim-zero}. 
Il vient $\dim_y Y_x=n-d$. 
\end{proof}

\subsection{}

Soit $X$ un espace $k$-analytique. Nous noterons $X^\#$ l'ensemble des points $x$ de $X$ tels que 
$d_k(x)=\dim_x X$. Si $x$ est un point de $X$ et si $V$ est un domaine analytique de $X$ contenant 
$x$ on a $\dim_x V=\dim_x X$, si bien que $V^\#=X^\#\cap V$. 

Soit $x\in X^\#$. Soit $V$ un domaine analytique de $X$ et soit $Y$ un fermé de Zariski de $V$
contenant $x$. Comme $d_k(x)=\dim_x X$ on a $\dim_x Y\geq \dim_x X$ et donc $\dim_x Y=
\dim_x  X=\dim_x V$. Ainsi
l'adhérence de Zariski $\overline{\{x\}}^{V_{\text{\upshape Zar}}}$ est de dimension $\dim_x V$, ce qui entraîne que
$\overline{\{x\}}^{V_{\text{\upshape Zar}}}$ est une composante irréductible de $V$ de dimension $\dim_x X$, et c'est alors 
nécessairement la seule composante irréductible de $V$ contenant $x$. 

\begin{lemm}\label{lemm.abhyankar-finitude}
Soit $x$ un point de $X^\sharp$.

\begin{enumerate}
\item L'extension graduée
$\hrt x$ de~$\widetilde k$ est de type fini. 
\item L'extension $\hrt x^1$ de $\widetilde k^1$ 
et le groupe abélien $\abs{\hr x^\times}/\abs{k^\times}$
sont de type fini.
\item Supposons $k$ stable. Le corps valué complet $\hr x$ est alors également stable. 
\end{enumerate}
\end{lemm}
\begin{proof}
On peut supposer $X$ affinoïde de dimension $n=\dim_x(X)$.
Soit $f_1,\dots,f_n\in\mathscr O(X)$ telles
que les $f_i(x)$ soient non nulles 
et que $(\widetilde{f_1(x)},\dots,\widetilde{f_n(x)})$ soit une base de transcendance
graduée de $\hrt x$ sur $\widetilde k$. 
Soit $f=(f_1,\dots,f_n)\colon X \to \A^n_k$ le morphisme induit par
cette famille.
Par construction, $f(x)=\eta_r$, où $r=\abs{f(x)}$.
D'après l'exemple \ref{exem-ktr-tilde}, et au vu de la densité
de $k(T_1,\ldots, T_n)$ dans $\hr{\eta_r}$, le
corps gradué $\hrt{\eta_r}$ est une extension gradué transcendante pure de degré de transcendance $n$
sur $\widetilde k$. Il s'ensuit que $d_k(\eta_r)=n$. On déduit alors du lemme
\ref{lemm.fibre-dim-zero} que $\dim f^{-1}(\eta_r)=0$. 
En particulier, $\hr x$ est fini sur $\hr{\eta_r}$. En conséquence $\hrt x$ est fini sur $\hrt{\eta_r}$, 
et $\hrt x$ est dès lors une extension de type fini de $\widetilde k$, puisque c'est le
cas de $\hrt{\eta_r}$ par sa description explicite que nous venons de rappeler. 
Ceci montre (a), et (b) s'en déduit par le corollaire \ref{coro-gradue-tf}.

Montrons maintenant (c). On suppose que $k$ est stable. Comme $\hr x$ est fini sur $\hr{\eta_r}$ 
il suffit, pour prouver que $\hr x$ est stable, de démontrer que $\hr {\eta_r}$ est stable. Par récurrence sur la longueur 
de $r$ on se ramène au cas où celle-ci vaut 1, et cela découle alors de \cite{temkin2010}, Thm. 6.3.1 (iii). 
\end{proof}

\subsection{Dimension centrale d'un germe d'espace $k$-analytique}
\label{ss.dimcent}
Soit $(X,x)$ un germe d'espace $k$-analytique. 
La \emph{dimension centrale} $\dim_{\rm cent}(X,x)$ de~$(X,x)$ est le minimum,
lorsque $V$ parcourt l'ensemble des voisinages analytiques de~$x$
dans~$X$,
de la  dimension de l'adhérence de Zariski de~$x$ dans~$V$.
Elle est minorée par~$d_k(x)$ et majorée par 
$\dim_x X$, et est dès lors égale à 
$\dim_x X$ si $x\in X^\sharp$. 

Supposons que le germe $(X,x)$ soit bon. D'après
le corollaire~3.2.9 de \cite{ducros2018}, on a alors
l'égalité
\begin{equation}
\label{formule.dimcent}
\dim_{\rm cent}(X,x)+\dim_{\rm Krull}(\mathscr O_{X,x})=\dim_x (X).
\end{equation}

\subsection{}
Soit $X$ un espace $k$-analytique et soit $x\in X^\sharp$. 
Soit $V$ un bon domaine analytique de $x$. 
Le point $x$ appartient alors
à $V^\sharp$, ce qui entraîne que 
$\dim_{\rm cent}(V,x)=\dim_x V$, puis que
$\mathscr O_{V,x}$ est de dimension de Krull nulle, ce qui veut
dire que c'est un anneau local artinien, et donc un corps si $X$ est réduit. 
Il s'ensuit que $x$ appartient au lieu de Cohen-Macaulay de $X$, et à son lieu régulier
si $X$ est réduit.

\begin{lemm}\label{lemm.voisinage-irreductible}
Soit $X$ un espace $k$-analytique et soit $x$ un point de~$X^\sharp$.
Il existe une base de voisinages de~$x$ formée de domaines
analytiques irréductibles. Lorsque $x$ possède un voisinage affinoïde,
il existe une base de voisinages de~$x$ formée de domaines
analytiques irréductibles.
\end{lemm}
\begin{proof}
Quitte à restreindre~$X$, il suffit de prouver qu'il existe
un tel voisinage. Soit $U$ le lieu de normalité de~$X_\red$;
c'est un ouvert de Zariski dense de~$X$ contenant~$x$.
On peut donc supposer que $X_\red$ est normal.
Tout voisinage analytique (resp.  voisinage affinoïde) connexe
de~$X$ est alors irréductible.
\end{proof}

\begin{lemm}\label{plat.si}
Soit~$n$ et~$d$ des entiers.
Soit $f\colon Y\to X$ un morphisme fini entre 
espaces~$k$-analytiques purement de
dimension~$n$. Supposons que~$X$ est régulier
et que~$Y$ vérifie la propriété~$(\mathrm S_d)$.
Soit $x$ un point de $X$
tel que $d_k(x)\geq n-d$. 
Le morphisme~$f$ est plat
au dessus d'un voisinage de $x$. 
\end{lemm}

\begin{proof} Le morphisme $f$ étant fini, il est topologiquement propre. 
Il suffit donc pur prouver l'assertion de montrer que le lieu de platitude de $f$ 
contient tous les antécédents de $x$. 
Pour établir ce fait, on peut remplacer $X$ par n'importe quel de ses domaines
affinoïdes contenant $x$, et donc supposer $X$ affinoïde ; dans ce cas, 
$Y$ l'est également puisque $f$ est fini. 
Soit $y$ un antécédent de $x$. 
La finitude de $f$ entraîne que les
les dimensions centrales de~$(X,x)$ et~$(Y,y)$ sont égales
(\cite{ducros2018}, 3.2.4) ; par ailleurs,
l'hypothèse faite sur~$x$ assure que la dimension centrale
de~$(X,x)$ vaut au moins~$n-d$. Comme
on a d'autre part supposé que~$\dim_x (X)=\dim_y (Y)=n$, la
formule~\eqref{formule.dimcent}
ci-dessus implique que
\[\dim_{\text{Krull}}(\mathscr O_{X,x})
=\dim_{\text{Krull}}(\mathscr O_{Y,y})\leq d.\]
Puisque~$\mathscr O_{Y,y}$ satisfait la
propriété~$(\mathrm S_d)$, il est de Cohen-Macaulay. 
La $\mathscr O_{X,x}$-algèbre finie locale~$\mathscr O_{Y,y}$ est de Cohen-Macaulay,
et de même dimension
de Krull que l'anneau local régulier $\mathscr O_{X,x}$.
En conséquence 
$\mathscr O_{Y,y}$ est une $\mathscr O_{X,x}$-algèbre plate
(\cite{ega4.2}, EGA IV\textsubscript 2, Proposition~6.1.5), 
et le morphisme sans bord $f$ est dès lors plat en $y$. 
\end{proof}

\begin{coro}\label{plat.cohen}
Soit~$n$ un entier, 
soit~$f\colon  Y\to X$ un morphisme fini entre bons 
espaces~$k$-analytiques purement de
dimension~$n$. Supposons que $X$ est régulier
et que $Y$ est de
Cohen-Macaulay. Le morphisme~$f$ est alors plat.
\end{coro}

\begin{lemm}\label{lemm.K-inter-U}
Soit $X$ un espace $k$-analytique, soit $U$ un ouvert de Zariski de~$X$
et soit $K$ une partie compacte de~$X$contenue dans $X^\sharp$. 
Alors $K\cap U$ est compact.
\end{lemm}
\begin{proof}
Soit $(X_i)_{i\in I}$ la famille des composantes irréductibles de~$X$;
pour tout~$i$, soit $X'_i$ le complémentaire dans~$X_i$
de la réunion des autres composantes.

Soit $F$ un fermé de Zariski strict de~$X_i$.
Pour tout $x\in F$, on a $d_k(x)\leq \dim(F)<\dim(X_i)\leq \dim_x(X)$;
par conséquent, $K$ ne rencontre pas~$F$.

On en déduit d'abord que 
si $i$ et~$j$ des indices distincts dans~$I$, le compact $K$ ne rencontre
pas~$X_i\cap X_j$.
Autrement dit, $K$ est la réunion disjointe de ses parties ouvertes
$K\cap X'_i$; il en résulte
que $X'_i\cap K$ est compact pour tout~$i$, et vide pour presque tout~$i$.

On écrit alors $K\cap U=\bigcup_{i\in I} K\cap U\cap X'_i$.
Soit $i\in I$. Si $U\cap X_i=\emptyset$, alors $K\cap U\cap X'_i=\emptyset$.
Sinon, $U\cap X_i$ est un ouvert de Zariski dense de~$X_i$
et son complémentaire dans~$X_i$ ne rencontre pas~$K$, si bien
que $K\cap U\cap X'_i=K\cap X'_i$.

Cela prouve que $K\cap U$ est compact.
\end{proof}

\section{Morphismes finis et plats}

\subsection{Morphismes finis et plats : le cas global}\label{ss-degfiniplat-global}
Soit $f\colon Y\to X$ un morphisme fini d'espaces $k$-analytiques. 
Le morphisme $f$ est plat si et seulement le faisceau cohérent $f_*\mathscr O_Y$ est un $\mathscr O_X$-module 
G-localement libre, et localement libre si $X$ est bon : c'est une conséquence de la proposition 4.3.1 de \cite{ducros2018}, 
du fait que la platitude de $f$ peut se tester G-localement sur $X$, et du fait que dans le cas bon, un faisceau cohérent est G-localement libre
si et seulement s'il est localement libre. 

Supposons $f$ plat. Le faisceau G-localement libre $\mathscr F:=f_*\mathscr O_Y$ possède alors un rang sur $\mathscr O_X$, qui est une
fonction localement constante sur $X$
à valeurs entières, que l'on note $x\mapsto \deg_x (f)$. On dit que $\deg_x (f)$ est le degré de $f$ au-dessus de $x$.
Un point $x$ appartient à $f(Y)$ si et seulement si $\deg^x (f)\neq 0$ ; il en résulte que $f$ est ouvert. 

Soit $x\in X$ et soit  $y_1,\ldots, y_r$ les antécédents de $x$ sur $Y$. Soit $U$ un bon domaine 
analytique de $X$ contenant $x$, et soit $V$ l'image réciproque $f^{-1}(U)$ ; 
c'est un bon domaine analytique de $Y$ contenant $y$. Notons $\mathscr F\otimes \mathscr O_{U,x}$ 
la fibre en $x$ de $\mathscr F|_U$ vu comme faisceau sur l'espace topologique $U$. 
On a par propreté topologique 
$\mathscr F\otimes \mathscr O_{U,x}=\prod \mathscr O_{V,y_i}$. Chacun des facteurs $\mathscr O_{V,y_i}$ est libre
de rang fini $r_i$ sur$\mathscr O_{U,x}$, et $\deg_x f=\sum_i r_i$. 

\subsection{}\label{ss.deg-x-fini-plat}
Soit $f\colon Y\to X$ un morphisme 
d'espaces $k$-analytiques. 

Soit $\Sigma$ une partie connexe et non vide de~$X$
telle que $f$ soit fini et plat au-dessus d'un voisinage de~$\Sigma$.
Le degré de~$f$ au-dessus d'un point~$x$ de~$\Sigma$ ne dépend
alors pas de~$x$ et sera appelé le degré de~$f$ au-dessus de~$\Sigma$.
Ce degré est nul si et seulement s'il existe un point de $\Sigma$ qui n'appartient pas à 
$f(Y)$ ; et si c'est le cas, $\Sigma$ possède un voisinage ouvert ne rencontrant pas $f(Y)$. 

\subsection{Morphismes finis et plats : le cas local}\label{ss-degfiniplat-local}
Soit $f\colon Y\to X$ un espace $k$-analytique et soit $y$ un point de $Y$ ; posons $x=f(y)$. 
Supposons $f$ fini et plat en $y$. Le lieu de platitude de $f$ étant un ouvert (de Zariski) de $Y$, 
il existe un voisinage analytique compact $U$ de $x$ dans $X$ et un voisinage analytique compact $V$ de $y$ dans $f^{-1}(U)$
tels que $f$ induise un morphisme fini et plat de $V$ sur $U$. Par propreté topologique et parce que $y$
est isolé dans $f^{-1}(x)$, on peut restreindre $V$ et $U$
de sorte que $y$ soit le seul antécédent de $x$ sur $V$. Dans ce cas le degré de $f|_V$ au-dessus de $x$
ne dépend plus du choix de $V$ et $U$. Il est appelé le degré de $f$ en $y$ et est noté $\deg^y(f)$. Si $U'$ est un bon domaine
analytique de $U$ contenant $x$ et si l'on pose $V'=f^{-1}(U')\cap V$ alors $V'$ est un bon domaine analytique de $V$ contenant $y$
et $\mathscr O_{V',y}$ est libre de rang $\deg^y(f)$ sur $\mathscr O_{U',x}$. 

Si $f$ est fini et plat globalement et si $y_1,\ldots, y_r$ sont les antécédents de $x$ sur $Y$ on a 
$\deg_x(f)=\sum_i \deg^{y_i}(f)$. 

\begin{lemm}\label{lemm-finiplat-scs}
Soit $f\colon Y\to X$ un morphisme  d'espaces $k$-analytiques, 
que l'on suppose fini et plat en tout point de $Y$. 
La fonction $y\mapsto \deg^y(f)$ est semi-continue supérieurement.
\end{lemm}

\begin{proof}
Soit $y\in Y$ et soit $x$ son image sur $X$. 
Puisque $f$ est fini et plat en $y$, il existe un voisinage analytique compact $U$ de $x$ et un voisinage
analytique compact $V$ de $y$ dans $f^{-1}(U)$ tel que $f^{-1}(x)\cap V=\{y\}$ ensemblistement, 
et tel que $f$ induise un morphisme fini et plat de degré constant $\deg^y(f)$ de $V$ sur $U$. 
Soit $\eta$ un point de $V$ et soit $\xi$ son image sur $U$. On a alors
$\deg^y(f)=\deg_\xi(f)=\sum_{\zeta \in f^{-1}(\xi)\cap V} \deg^\zeta(f)$. Il s'ensuit que 
$\deg^\zeta(f)\leq \deg^y(f)$ pour tout $\zeta \in f^{-1}(\xi)\cap V$ ; en particulier
$\deg^{\eta}(f)\leq \deg^y(f)$. 
\end{proof}

\begin{lemm}\label{lemme.finiplat-local}
Soit $f\colon Y\to X$ un morphisme d'espaces $k$-analytiques
de dimensions~$\leq n$.
Soit $y$ un point de~$\Int(Y/X)$. 

\begin{enumerate}
\item
Si $f$ est de dimension relative nulle en~$x$, alors
le morphisme~$f$ est fini en~$y$.

\item
Si $d_k(f(y))=n$, 
alors le morphisme~$f$ est fini en~$y$.

\item
Si $d_k(f(y))=n$ et si $X$ est réduit en~$f(y)$, 
alors $f$ est fini et plat en~$y$.
\end{enumerate}
\end{lemm}
\begin{proof}
\begin{enumerate}
\item 
C'est le corollaire~3.1.10 de~\cite{berkovich1993}.

\item
D'après le lemme~\ref{lemm.fibre-dim-zero},
la fibre~$f^{-1}(x)$ est de dimension nulle.
L'assertion découle donc du~\emph a).

\item
Compte tenu de ce qui précède,
il suffit de vérifier que $f$ est plat en~$y$. Posons $x=f(y)$.
Fixons un domaine affinoïde~$V$ de~$X$ contenant~$x$
et soit $W$ son image réciproque dans~$Y$.
Comme $d_k(x)=n$, l'anneau~$\mathscr O_{V,x}$ est artinien;
comme $X$ est réduit en~$x$, cet anneau est réduit;
c'est donc un corps si bien que $\mathscr O_{W,y}$ est plat 
sur~$\mathscr O_{V,x}$. Par suite, $f$ est plat en~$y$.
\end{enumerate}
\end{proof}

\begin{lemm}\label{lemme.finiplat-global}
Soit $f\colon Y\to X$ un morphisme d'espaces $k$-analytiques
de dimensions~$\leq n$.
Soit $x$ un point de~$X$. On suppose  que 
$f^{-1}(x)$ ne rencontre pas $\partial(Y/X)$ 
et que $f$ est compact au-dessus d'un voisinage de~$x$.

\begin{enumerate}
\item
Si $f^{-1}(x)$ est de dimension nulle , alors
le morphisme~$f$ est fini au-dessus d'un voisinage de~$x$.

\item
Si $d_k(x)=n$, 
alors le morphisme~$f$ est fini au-dessus d'un voisinage de~$x$.

\item
Si $d_k(x)=n$ et si $X$ est réduit en~$x$, 
alors $f$ est fini et plat au-dessus d'un voisinage de~$x$.
\end{enumerate}
\end{lemm}
\begin{proof}
\begin{enumerate}
\item 
La fibre $f^{-1}(x)$ est discrète et compacte,
elle est donc finie. Notons $f^{-1}(x)=\{y_1,\dots,y_r\}$.

Soit $i\in\{1,\dots,r\}$. D'après le lemme~\ref{lemme.finiplat-local},
\emph a),
le morphisme~$f$ est en fini en chaque~$y_i$.
Il existe donc un voisinage~$V_i$ de~$y_i$
et un voisinage~$U_i$ de~$x$ tels que $f$ induise
un morphisme fini de~$V_i$ sur~$U_i$. 
La	compacité de $f$ permet de supposer ces
voisinages~$V_i$ deux à deux disjoints, 
et
elle assure également l'existence d'un voisinage
ouvert $U$ de~$x$ contenu dans l'intersection des~$U_i$
tel que $f^{-1}(U)$
soit contenu dans la réunion des~$V_i$.
Le morphisme de $f^{-1}(U)$ sur $U$
induit par $f$ est alors fini, d'où l'assertion. 

\item
D'après le lemme~\ref{lemm.fibre-dim-zero},
la fibre~$f^{-1}(x)$ est de dimension nulle.
L'assertion découle donc du~\emph a).

\item
Avec les notations précédentes,
il reste à vérifier que $f$ est plat en tout point $y\in f^{-1}(x)$,
ce qui est l'assertion~\emph c) du lemme~\ref{lemme.finiplat-local}.
\end{enumerate}
\end{proof}

\section{Réduction graduée des germes d'espaces
analytiques}

Nous nous proposons ici de rappeler les points
essentiels de la théorie 
de la réduction graduée des germes d'espaces
analytiques que Temkin
a développée dans \cite{temkin2004}, et qui généralisait
et simplifiait sa théorie de la réduction classique
des germes d'espaces strictement $k$-analytiques
mise au point dans \cite{temkin2000}.

\subsection{Réduction à la Temkin : le cas d'un espace affinoïde}\label{ss-reduction-temkin}
Soit $A$ une algèbre $k$-affinoïde. 
Nous désignerons toujours par $\widetilde A$ la réduction graduée de $A$
relative à la semi-norme spectrale $\rho$ de $A$ (\ref{ss-red-graduee}). 
C'est une $\widetilde k$-algèbre graduée, et elle est de type
fini. Pour le voir, choisissons un morphisme surjectif admissible $f\colon
k\{T/r\}\to A$. La proposition 3.1 de \cite{temkin2004} assure que $f$
induit un morphisme fini entre les réductions graduées de ses source et but. 
Par ailleurs, la réduction graduée de $k\{T/r\}$ s'identifie à l'algèbre de
polynômes $\widetilde k[T/r]$, en vertu de l'exemple \ref{exem-ktr-tilde}
et de la densité de $k[T]$ dans $k\{T/r\}$, d'où notre assertion. 

Posons $X=\mathscr M(A)$. 
Notons $\widetilde X$ le spectre homogène de $\widetilde A$, c'est-à-dire l'ensemble
de ses idéaux premiers homogènes. Si $\xi$ est un élément de $\widetilde X$ nous noterons
$\kappa(\xi)$ le corps gradué des fractions de $\widetilde A/\xi$, et $f\mapsto f(\xi)$ l'application 
naturelle de $A$ vers $\kappa(\xi)$. On munit $\widetilde X$ de la topologie de Zariski, 
dont une base d'ouverts est formée des $D(a):=\{\xi \in  \widetilde X, 
a(\xi)\neq 0\}$ où $a$ parcourt l'ensemble des éléments homogènes de $\widetilde A$. 
Puisque $\widetilde A$ est un anneau gradué noethérien, l'espace topologique
$\widetilde X$
est noethérien.

On dispose d'une application dite de réduction ou spécialisation $\pi \colon X\to
\widetilde X$, qui envoie un point $x$ sur le noyau du morphisme de $\widetilde A$
vers $\hrt x$ induit par l'évaluation en $x$. Cette application est anti-continue, c'est-à-dire que l'image
réciproque de tout ouvert est fermée. En effet, il suffit 
par quasi-compacité de vérifier que $\pi^{-1}(D(\widetilde f))$ est fermé pour toute $f\in A$. Or
si $\rho(f)=0$ alors $\widetilde f=0$  et  $\pi^{-1}(D(\widetilde f))=\emptyset$, et sinon 
 $\pi^{-1}(D(\widetilde f))$ est le lieu de validité de l'égalité $\abs f=\rho(f)$. 
 L'application $\pi$ est surjective, et chaque point générique d'une composante irréductible
 de $\widetilde X$ a un unique antécédent sur $X$ (\cite{temkin2004}, Prop. 3.3). 
 
L'inclusion $\widetilde A^1\hookrightarrow \widetilde A$ induit une application continue
de $\widetilde X$ vers $\widetilde X^1:=\Spec \widetilde A^1$, qui composée avec $\pi$ donne lieu à l'application 
de spécialisation classique de $X$ vers $\widetilde X^1$ (qui n'est en pratique utilisée que dans le cas
strictement $k$-affinoïde). 
Le morphisme naturel de $\widetilde A^1\otimes_{\widetilde k_1}\widetilde k$ 
dans $\widetilde A$ est injectif et a pour image $\bigoplus_{r\in \abs{\widetilde k^\times}}
\widetilde A^r$ (\ref{ss-tenseurk1-kr}). Il est donc bijectif si et seulement si l'ensemble
des degrés des éléments homogènes non nuls de $\widetilde A$ est égal à $\abs{k^\times}$, 
ce qui revient à demander que $\rho(A)=\abs k$. Supposons que ce soit le cas (ce qui entraîne que $A$
est strictement $k$-affinoïde). L'application continue 
$\widetilde X\to \widetilde X^1$ est alors un homéomorphisme, et modulo cette identification $\pi$
est la réduction classique. 

Le lemme suivant est bien connu, mais nous en donnons la preuve faute de référence explicite dans
la littérature. 

\begin{lemm}\label{lemme-adherence-tube}
Soit~$X$ un espace $k$-affinoïde et soit~$\pi \colon X\to \widetilde X$ l'application de spécialisation. 
Soit~$\xi\in \widetilde X$. L'adhérence de $\pi^{-1}(\xi)$ dans~$X$ est la réunion des~$\pi^{-1}(\eta)$
où~$\eta$ parcourt l'ensemble des générisations de~$\xi$ dans~~$\widetilde X$.
\end{lemm}

\begin{proof}
Soit~$E$ la réunion des~$\pi^{-1}(\eta)$
où~$\eta$ parcourt l'ensemble des générisations de~$\xi$ dans~$\widetilde X$.On peut également décrire 
$E$ comme l'intersection des~$\pi^{-1}(U)$ où~$U$ parcourt l'ensemble des ouverts de Zariski 
de~$\widetilde X$ contenant~$\xi$. Par conséquent~$E$ est un fermé de~$X$ contenant~$\pi^{-1}(\xi)$
et donc~$\overline{\pi^{-1}(\xi)}$. Réciproquement, considérons un point ~$y$ de~$E$ et montrons
que~$y\in  \overline{\pi^{-1}(\xi)}$. Posons~$\eta=\pi(y)$ ; par définition de~$E$, le point~$\xi$ est une
spécialisation de~$\eta$. 
Soit~$V$ un voisinage affinoïde de~$y$ dans~$X$ et soit~$\zeta$ la spécialisation de~$y$ sur~$\widetilde V$. 
L'inclusion~$f$ de~$V$ dans~$X$ est intérieure en~$y$. En vertu de la proposition 3.4 de \cite{temkin2004}, cela
signifie que le morphisme~$\widetilde f\colon \widetilde V\to
\widetilde X$ induit un morphisme fini du sous-schéma fermé
réduit~$\overline{\{\zeta\}}$ de~$V$
sur 
$\widetilde X$. L'image~$\widetilde f(\overline{\{\zeta\}})$ est donc fermée
dans~$\widetilde X$ (par la variante gradué du lemme de \textit{going-up}, qui se
démontre par réduction au cas classique via un changement de base de~$\widetilde k$ à~$\widetilde k(T/\rho)$
pour un polyrayon~$\rho$ déployant la situation). Or
l'image de~$\zeta$ sur~$\widetilde X$
est égale à~$\eta$ ; ce dernier est donc dense dans
le fermé~$\widetilde f(\overline{\{\zeta\}})$ de~$\widetilde X$,
qui est dès lors égal à~$\overline{\{\eta\}}$ et contient en particulier $\xi$. 
Ainsi ce dernier possède un antécédent $\omega$ par~~$\widetilde f$. Par surjectivité de la spécialisation
il existe un élément $v$ de $V$ se spécialisant sur $\omega$ ; mais l'image de~$v$ sur~$\widetilde X$ est alors
égale à~$\widetilde f(\omega)$, c'est-à-dire à $\xi$. Par conséquent $V$ rencontre~$\pi^{-1}(\xi)$
et~$x$ appartient à~$\overline{\pi^{-1}(\xi)}$.
\end{proof}

\subsection{Réduction à la Temkin : le cas des germes}
Soit $X$ un espace $k$-analytique 
séparé\footnote{La réduction de Temkin existe aussi
dans le cas non séparé, mais est alors plus délicate à définir et
à manipuler, et nous ne nous en servirons que dans le cas séparé ;
aussi avons-nous choisi de ne la présenter que dans ce cadre.} et soit $x\in X$.  
Temkin a
défini dans \cite{temkin2004}, 
la \emph{réduction} 
$\widetilde{(X,x)}$ de $X$ ; 
c'est un ouvert quasi-compact et non vide
de l'espace de Zariski-Riemann 
$\zr {\hrt x}{\tilde k}$ de $\hrt x$ sur
$\tilde k$.

Rappelons rapidement sa définition. 
Supposons tout d'abord que $X$ est affinoïde, d'algèbre 
associée $A$. 
Soit $R$ l'image de $\widetilde A$ dans $\hrt x$, par le morphisme induit par
l'évaluation, et soit $R^{\mathrm{hom}}$ l'ensemble des éléments
homogènes de $R$ ; l'anneau gradué~$R$
est une sous-$\widetilde k$-algèbre
graduée de type fini de $\hrt x$. 
Temkin définit alors $\widetilde{(X,x)}$ comme le sous-ensemble $\zr{\hrt x}{\widetilde k}
\{R^{\mathrm{hom}}\}$ de $\zr{\hrt x}{\widetilde k}$. Ce sous-ensemble est égal à $\zr{\hrt x}{\widetilde k}\{E\}$
pour n'importe quel ensemble fini $E$ de générateurs homogènes de $R$ 
sur~$\widetilde k$ ; c'est donc un ouvert quasi-compact et non vide 
de  $\zr{\hrt x}{\widetilde k}$ qui ne dépend que du germe affinoïde $(X,x)$ 
(voir le début de la section~4 de \cite{temkin2004}) 
et qu'on note $\widetilde {(X,x)}$. 

Si l'on ne suppose plus que $X$ est affinoïde, on choisit une famille finie $(V_i)$ de domaines affinoïdes de $X$
contenant $x$ et dont la réunion est un voisinage de $x$ dans $X$. On démontre
(\cite{temkin2004}, Prop. 4.1) que la réunion des
$\widetilde{(V_i,x)}$ ne dépend que du germe $(X,x)$, et c'est elle que l'on note $\widetilde {(X,x)}$. 
Il résulte de \loccit et du théorème 4.5 de \cite{temkin2004} que $(V,x)\mapsto \widetilde{(V,x)}$ est une application 
croissante, commutant aux unions et intersections finies, qui établit une bijection entre l'ensemble des (germes
de) domaines analytiques de $(X,x)$ et celui des ouverts quasi-compacts et non vides de $\zr{\hrt x}{\widetilde k}$. 

Soient $f_1,\ldots, f_n$ des éléments de $\mathscr O_{X,x}^\times$ ; pour tout $i$, posons $r_i=\abs{f_i(x)}$. 
Soit $V$ le domaine analytique de $(X,x)$ défini par les inégalités $\abs {f_i}\leq r_i$. Alors 
$\widetilde {(V,x)}$ est égal à $\widetilde {(X,x)}\cap \zr{\hrt x}{\widetilde k}\{\widetilde {f_i(x)}\}_i$;
en effet, on se ramène aussitôt au cas affinoïde, auquel cas c'est le lemme 4.4 de \cite{temkin2004}. 

Le germe $(X,x)$ est bon si et seulement si $\widetilde{(X,x)}$ est affine (\cite{temkin2004}, Thm. 5.1; pour la définition 
d'un ouvert affine dans ce contexte, voir \ref{def-affine-zr}) ; il est sans bord si et seulement si
$\widetilde{(X,x)}=\zr{\hrt x}{\widetilde k}$ (\cite{temkin2004}, Thm. 5.2). 

\begin{rema}\label{normezr}
Soit $(X,x)$ un germe d'espace $k$-analytique, soit $f$ une fonction 
sur $(X,x)$ et soit $(Y,x)$ un domaine analytique de $(X,x)$ 
tel que $\abs f$ soit identiquement égale à $\abs{f(x)}$ 
sur $(Y,x)$ ; comme $(Y,x)$ est contenu dans le domaine 
de $(X,x)$ défini par l'égalité $\abs f=\abs{f(x)}$, on a
\[ \widetilde{(Y,y)}\subset  \widetilde{(X,x)}\cap \zr {\hrt 
x}{\tilde k}\{\widetilde{f(x)}, \widetilde{f(x)}^{-1}\}.\]
 Par conséquent, $\widetilde{f(x)}$ est une section inversible de 
$\mathscr O_{\zr {\hrt x}{\tilde k}}(\widetilde{Y,y})$. 
\end{rema}

\subsection{Le cas strictement affinoïde}\label{ss-redgerm-strictaff}
Soit $X$ un espace strictement $k$-affinoïde d'algèbre des fonctions $A$ et soit $x$ un point de $X$. 
Soit $f$ un élément de $A$ inversible en $x$. La semi-norme spectrale $\rho(f)$ appartient à $\abs{k^\times}^\Q$ ; 
par conséquent, si $\abs{f(x)}\notin \abs{k^\times}^\Q$ on a $\abs{f(x)}<\rho(f)$ et l'image de $\widetilde f$
dans $\hrt x$ est nulle. Il s'ensuit que $\widetilde{(X,x)}$ est de la forme $\zr{\hrt x}{\widetilde k}\{\lambda_1,\ldots, \lambda_n\}$ où les
$\lambda_i$ sont des éléments de $\hrt x^\times$ tels que $\abs{\lambda_i}\in \abs{\widetilde k^\times}^\Q$
pour tout $i$. Par conséquent 
$\widetilde{(X,x)}$ est 
l'image réciproque d'un ouvert affine de $\zr{\hrt x^1}{\widetilde k^1}$
(\ref{def-strict-zr}). 

\subsection{Le cas où $\widetilde{(X,x)}$ est affine et strict}
\label{ss-redgerm-strictaff}
Si $X$ un espace $k$-analytique séparé et soit $x\in X$. 
Il résulte de \ref{ss-redgerm-strictaff}
que si $(X,x)$ possède un représentant strictement $k$-affinoïde,
$\widetilde{(X,x)}$ est l'image réciproque d'un ouvert affine de $\zr{\hrt x^1}{\widetilde k^1}$.

Supposons réciproquement que $\widetilde{(X,x)}$ soit
l'image réciproque d'un ouvert affine de $\zr{\hrt x^1}{\widetilde k^1}$. 
Il résulte de la preuve du lemme 3.5.1 de \cite{ducros2018} qu'il existe
un entier $n$, un point $y$ du polydisque unité $D$
de dimension $n$ 
et un morphisme sans bord
de $(X,x)$ vers $(D,y)$. Lorsque la valeur absolue
de $k$ n'est pas triviale, $y$ possède une base de voisinages strictement
affinoïdes dans $D$, si bien que $(X,x)$ possède un représentant
strictement affinoïde. 

\subsection{}\label{ss-redstrict-nontrivial}
Supposons que la valeur absolue de $k$ n'est pas triviale. 
On déduit alors des résultats rappelés au \ref{ss-redgerm-strictaff}
que 
$(X,x)$ est strictement $k$-analytique si et seulement
si $\widetilde {(X,x)}$ est strict, \cf \cite[Lemma 3.5.2]{ducros2018}. 
Mentionnons incidemment que si $(X,x)$ est à la fois bon et strict,
$\widetilde{(X,x)}$ est à la fois affine et strict, donc 
est l'image réciproque d'un ouvert affine de $\zr{\hrt x^1}{\widetilde k^1}$ 
(\ref{def-strict-zr}) ; il découle alors de
\ref{ss-redgerm-strictaff}
que le 
germe $(X,x)$ est 
strictement affinoïde. Autrement dit tout point d'un bon espace strictement 
$k$-analytique admet un voisinage strictement affinoïde
ce qui, à notre connaissance, ne figurait pas jusqu'ici explicitement dans la littérature. 

Notons que l'hypothèse que la valeur absolue de $k$ n'est pas triviale est indispensable à la validité
des énoncés ci-dessus. Supposons en effet que $\abs{k^\times}=\{1\}$, 
et soit $X$ l'ouvert 
de $\A^1_k$ défini par les inégalités $0<\abs T<1$ ; il est sans bord, si bien que pour tout $x\in X$ la réduction 
$\widetilde{(X,x)}$ est égale à $\zr {\hrt x}{\widetilde k}$, et est en particulier
stricte ; mais comme  $X$ n'est constitué que de points de type $3$, il ne possède aucun domaine
strictement affinoïde non vide. 

\section{Espaces stricts et réduction de Temkin ; le cas trivialement valué} 
\subsection{Modification de la terminologie dans le cas trivialement valué}
\label{terminologie-stricte}
Nous aurons besoin dans la suite de manipuler des espaces analytiques (séparés) 
sur un corps trivalement valué dont tous les germes sont à réduction stricte. 
Pour éviter l'introduction d'un vocabulaire spécifique à cette situation, 
nous allons nous écarter de la terminologie standard. Nous appellerons désormais espace strictement $k$-analytique tout espace
$k$-analytique $X$ possédant un G-recouvrement $(X_i)$ par des domaines analytiques séparés tels que 
$\widetilde{(X_i,x)}$ et $\widetilde{(X_i\cap X_j,x)}$ soient stricts pour tout $(i,j)$
et tout $x\in X_i\cap X_j$ ; si $X$ est séparé (ce qui sera essentiellement toujours
le cas ici) cela revient à demander que $\widetilde{(X,x)}$ soit strict pour tout $x\in X$. 

Notre notion d'espace strictement $k$-analytique coïncide avec la notion usuelle lorsque $\abs{k^\times}\neq \{1\}$, mais en diffère si $k$
est trivialement valué ; sous cette hypothèse les espaces stricts
au sens traditionnel le sont dans le nôtre, mais la réciproque
est fausse en général : par exemple l'ouvert $0<\abs T<1$  de $\A^1_k$ est strict dans notre sens, mais pas dans l'acception classique. 

Les deux notions sont toutefois équivalentes dans le cas affinoïde, 
comme en atteste le lemme ci-dessous, pour lequel il n'est pas nécessaire
de supposer $\abs{k^\times}\neq \{1\}$.

\begin{lemm}
Soit $A$ une algèbre $k$-affinoïde. Elle est strictement 
$k$-affinoïde si et seulement si l'espace $\mathscr M(A)$ est 
strict au sens défini ci-dessus. 
\end{lemm}

\begin{proof}
Posons $X=\mathscr M(A)$. Si $A$ est strictement $k$-affinoïde alors
$X$ est strictement $k$-analytique au sens classique,
et \textit{a fortiori} au nôtre. Réciproquement, supposons que $X$
soit strict dans notre sens. Pour montrer que $X$ est strict au sens classique 
il suffit de s'assurer que pour tout $f\in A$ la semi-norme spectrale
$\norm f_\infty$  appartient 
à $\{0\}\cup\abs{k^\times}^\Q$. 
Soit $f\in A$ 
un élément tel que $\norm f_\infty\neq 0$ ; nous allons
montrer que $\norm f_\infty\in \abs{k^\times}^\Q$. 
Choisissons  $x\in X$ tel 
que $\abs {f(x)}=\norm f_\infty$.
L'image de l'élément $\widetilde f$ de $\widetilde A$
dans $\hrt x$ est alors égale à 
$\widetilde{f(x)}$, et l'ouvert quasi-compact
strict 
$\widetilde{(X,x)}$ de $\zr{\hrt x}{\widetilde k}$ 
est alors contenu dans $\zr{\hrt x}{\widetilde k}\{\widetilde{f(x)}\}$. 
D'après \ref{def-strict-zr}, ceci entraîne que $\abs{\widetilde{f(x)}}$ appartient
à $\abs{\widetilde k^\times}^\Q$, c'est-à-dire que $\abs{f(x)}\in
\abs{k^\times}^\Q$. 
\end{proof}

Nous n'emploierons désormais l'adjectif «strict» que dans notre sens, 
sauf mention expresse du contraire. Nous dirons qu'un germe est strict s'il admet
un représentant strict. 

\begin{lemm}
Soit $(X,x)$ un germe séparé d'espace $k$-analytique. Le
germe ${(X,x)}$ est strict si et seulement si sa réduction $\widetilde {(X,x)}$ est stricte.
\end{lemm}

\begin{proof}
L'implication
directe découle des définitions ; montrons
maintenant la réciproque (qui est déjà connue dans le cas où la valeur absolue de $k$ n'est pas triviale). 
Supposons donc que $\widetilde{(X,x)}$ est strict. Il s'écrit alors comme une union finie $\bigcup U_i$ 
d'ouverts affines et stricts. Chaque $U_i$ est de la forme $\widetilde{(X_i,x)}$ où $X_i$ est un domaine
analytique de $X$ contenant $x$. Il suffit alors de démontrer que chacun des $(X_i,x)$ possède un représentant
strict $(Y_i,x)$ : on pourra alors en effet toujours supposé quitte à restreindre les $Y_i$ qu'ils sont fermés
dans $X$, et la réunion des $Y_i$ constituera alors un représentant strict de $(X,x)$. On s'est ainsi ramené au cas
où $\widetilde{(X,x)}$ est affine. On a alors vu au \ref{ss-redgerm-strictaff} qu'il existe un entier $n$, un point
$y$ du polydisque unité $D$ de dimension $n$ et un morphisme sans bord $f$ de $(X,x)$ vers $(D,y)$. On peut
supposer quitte à remplacer $X$ par un voisinage ouvert convenable de $x$ que ce morphisme est défini
sur $X$ tout entier. Soit $z\in X$. Comme $f$ est sans bord, $\widetilde{(X,z)}$ est l'image réciproque de 
$\widetilde{(D,f(z))}$, mais ce dernier est strict puisque $D$ est strictement $k$-affinoïde. Par conséquent 
$\widetilde{(X,z)}$ est strict, et $X$ est strict.  
\end{proof}

\begin{rema}
Si $\widetilde {(X,x)}$ est strict et si la valeur absolue de $k$ n'est pas triviale, on peut trouver un représentant
de $(X,x)$ à la fois strict et compact, car tout espace strictement $k$-analytique admet une base de voisinages
stricts et compacts. Mais c'est faux en général lorsque $k$ est trivialement valué :
l'ouvert strict de $\A^1_k$ défini par les inégalités $0<\abs T<1$ ne possède aucun domaine analytique compact
strict hormis le vide. 
\end{rema}

\begin{lemm}\label{lemme-affine-redstricte}
Soit $X=\mathscr M(A)$ un espace $k$-affinoïde et soit $x\in X$. 
Le germe $(X,x)$ est strict si et seulement le degré de tout élément
homogène non nul de l'image du morphisme canonique de
$\widetilde A$ vers $\hrt x$ est de torsion modulo $\abs{k^\times}$. 
\end{lemm}

\begin{proof}
Soit $B$ l'image du morphisme canonique de
$\widetilde A$ vers $\hrt x$ et soit $S$ un ensemble fini de générateurs
homogènes non nuls de $B$ sur $\widetilde k$. La  réduction $\widetilde{(X,x)}$ est égale
à
$ \zr{\hrt x}{\widetilde k}\{S\}$. Si le degré de tout élément
homogène non nul $B$ est de torsion modulo $\abs{k^\times}$, c'est en particulier le cas
du degré de tout élément de $S$, et il s'ensuit que $\widetilde {(X,x)}$ est strict. 
Réciproquement, supposons que $\widetilde {(X,x)}$ soit strict. Il s'écrit alors
$\zr{\hrt x}{\widetilde k}\{E\}$ pour une certaine famille finie $E$ d'éléments
de $\hrt x^1$, et l'on a également
$\widetilde{(X,x)}=\zr{\hrt x}{\widetilde k}\{S\cup E\}$. 
Il résulte de l'égalité $\zr{\hrt x}{\widetilde k}\{S\cup E\}=\zr{\hrt x}{\widetilde k}\{E\}$
que $B$ est finie sur la sous-$\widetilde k$-algèbre graduée $C$ de $\hrt x$ engendrée par $E$
(\cite{temkin2004}, Lemma 2.2). L'algèbre graduée $B$
est en particulier contenue dans une extension finie du corps gradué
des fractions $\mathrm{Frac}(C)$. Le degré de tout élément homogène non nul de $B$ est donc de torsion 
modulo $\abs{\mathrm{Frac}(C)^\times}$, lequel est égal à $\abs{k^\times}$ par construction. 
\end{proof}

\subsection{}
Soit $(X,x)$ un germe séparé et strict d'espace $k$-analytique. 
Il existe alors un unique ouvert non vide
et quasi-compact de $\zr{\hrt x^1}{\widetilde k^1}$ 
dont l'image réciproque sur $\zr{\hrt x}{\widetilde k}$
est égale à $\widetilde{(X,x)}$ (c'est l'image directe de $\widetilde{(X,x)}$) ; 
on le note $\widetilde{(X,x)}^1$. Les faits suivants se déduisent de leur version graduée.

Si $X$ est strictement $k$-affinoïde
et si $A$ désigne l'algèbre des fonctions analytiques sur $X$,
il résulte des définitions que 
$\widetilde{(X,x)}^1=\zr{\hrt x^1}{\widetilde k^1}\{E\}$
où $E$ est un ensemble fini de générateurs
sur $\widetilde k^1$ de l'image
de la réduction classique $\widetilde A^1$ dans
$\hrt x^1$ ($\widetilde A^1$ est de type fini, et son image
dans $\hrt x^1$  l'est donc aussi).

L'application $(V,x)\mapsto \widetilde{(V,x)}^1$ est une application 
croissante, commutant aux unions et intersections finies, qui établit une bijection entre l'ensemble des (germes
de) domaines strictement $k$-analytiques de $(X,x)$ et celui des ouverts quasi-compacts et non vides de
$\zr{\hrt x^1}{\widetilde k^1}$. 

Le germe $(X,x)$ est bon si et seulement si $\widetilde{(X,x)_1}$ est affine ;
il est sans bord si et seulement si  $\widetilde{(X,x)^1}
=\zr{\hrt x^1}{\widetilde k^1}$. 

Soit $(f_1,\ldots, f_n)$ une famille d'éléments de $\mathscr O_{X,x}^\times$ tels que 
$\abs{f_i(x)}=1$. Pour tout $i$, soit $V_i$ le domaine analytique de $(X,x)$ défini par les inégalités $\abs {f_i}\leq 1$. Alors 
$\widetilde {(V,x)}^1=\widetilde {(X,x)}^1\cap \zr{\hrt x^1}{\widetilde k^1}\{\widetilde {f_i(x)}\}_i$. 

\begin{exem}\label{exemple-torsion-strict}
Soit $X$ un espace $k$-analytique et soit $x$ un point de $X$
tel que $\abs{\hr x^\times}$ soit de torsion modulo $\abs k^\times$. 
Il résulte alors de \ref{ss-zr-torsion}
que $\widetilde{(X,x)}$ est strict, et donc que $(X,x)$ est strict, et que
l'application continue naturelle de $\widetilde{(X,x)}$ sur  $\widetilde{(X,x)}^1$
est un homéomorphisme 
\end{exem}

\section{Réduction de Temkin et modèles formels}

\subsection{Schémas formels et leurs fibres génériques}
Soit $\mathfrak X$ un $k^\circ$-schéma formel plat possédant un recouvrement
localement fini par des ouverts affines topologiquement de présentation finie
sur $k^\circ$. 
On note $\mathfrak X_{\mathrm s}$ sa fibre spéciale $\mathfrak X\otimes_{k^\circ}
\widetilde k^1$ ; c'est un $\widetilde k^1$-schéma possédant un recouvrement ouvert localement fini 
par des $\widetilde k^1$-schémas de type fini. Précisons que nous autorisons toujours le corps $k$ à être trivialement
valué ; dans ce cas, on a $k=\widetilde k=\widetilde k^1$
et $\mathfrak X$ est simplement le $k$-schéma $\mathfrak X_s$, mais vu
comme schéma formel en munissant pour tout ouvert
$\mathfrak U$ de $\mathfrak X$ l'anneau $\mathscr O_{\mathfrak X}(\mathfrak U)$ de sa topologie
de limite des $\mathscr O_{\mathfrak X}(\mathfrak V)$ où $\mathfrak V$ parcourt l'ensemble des ouverts
affines de $\mathfrak U$ et où chaque $\mathscr O_{\mathfrak X}(\mathfrak V)$ est muni
de la topologie discrète. 

À  un tel schéma formel est associée sa fibre générique $\mathfrak X_\eta$. C'est un espace strictement
$k$-analytique (il l'est même au sens classique), 
paracompact et topologiquement séparé, séparé si $\mathfrak X$ est séparé, et compact 
si $\mathfrak X$ est quasi-compact. Si $\mathfrak X=\Spf A$ pour une certaine $k^\circ$-algèbre
plate et topologiquement de présentation finie $A$, on a $\mathfrak X_\eta=\mathscr M(A\otimes_{k^\circ}k)$,
et le cas général s'obtient en recollant ces constructions locales. 

Supposons toujours que $\mathfrak X=\Spf A$ et posons $A_k=A\otimes_{k^\circ}k$. 
L'anneau $A_k^\circ$ est alors la fermeture intégrale de $A$
dans $A_k$, et la réduction classique $\widetilde {A_k}^1$  est une $\widetilde k^1$-algèbre
de type fini. C'est en effet évident si $k$ est trivialement valué car alors $A=A_k=\widetilde {A_k}$. 
Supposons maintenant que $k$ n'est pas trivialement valué. Choisissons un morphisme continu et surjectif 
de $k^\circ$-algèbres
$k^\circ\{T\}\to A$, où $T$ est une famille finie d'indéterminées. 
Il induit un morphisme surjectif admissible
de $k$-algèbres de $k\{T\}$ vers $A_k$, 
et le théorème 6.3.5/1 de \cite{bosch-g-r1984} assure alors
que $A_k^\circ$ est entière sur $(k\{T\}^\circ=k^\circ \{T\}$,
c'est-à-dire sur son image, qui n'est autre que $A$ ; et il assure aussi que 
que $\widetilde {A_k}^1$ est fini sur $\widetilde{k\{T\}}^1=\widetilde k^1[T]$, 
et a fortiori de type fini sur $\widetilde k^1$. 
Et comme $A_k^\circ$ est 
l'ensemble des $f\in A
$ telles que $f(x)\in \hr x^\circ$ pour tout $x\in \mathscr M(A_k)$
et que 
que $\hr x^\circ$ est normal pour tout tel $x$, l'anneau $A_k^\circ$ est 
intégralement clos dans $A_k$ ; étant entier sur $A$, il est égal à la fermeture
intégrale de celui-ci dans $A$.

Les inclusions de $A$ dans $A_k^\circ$
et de $k^{\circ \circ}A$ dans $A_k^{\circ \circ}$ induisent un morphisme 
naturel de $A\otimes_k\widetilde k^1$ vers $\widetilde {A_k}^1$. Ce morphisme est entier 
puisque $A_k^\circ $ est entière sur $A$, et il est même en fait fini, 
puisque $\widetilde {A_k}^1$ est de type fini sur $\widetilde k^1$. 

En composant l'application de spécialisation classique 
de $\mathfrak X_\eta=\mathscr M(A_k)$ vers $\Spec (\widetilde {A_k}^1)$
avec le morphisme fini de $\Spec  (\widetilde {A_k})^1$
vers $\Spec A\otimes_k\widetilde k^1=\mathfrak X_s$, on obtient une application 
de spécialisation de $\mathfrak X_\eta$ vers $\mathfrak X_s$. Elle est anticontinue
par construction, et elle est surjective : on se ramène en effet par extension des scalaires
à démontrer que tout point fermé de $\mathfrak X_s$ appartient à son image, ce qui est
assuré par la proposition~8.3/8 de~\cite{bosch2014}. 

Lorsque $\mathfrak X$ n'est plus supposé affine, les constructions ci-dessus fournissent
encore par recollement une application de spécialisation $\pi$ de $\mathfrak X_\eta$ vers
$\mathfrak X_s$ qui est anti-continue et surjective. 
Par construction, cette application inuit pour tout $x\in X$ un $\widetilde k^1$-plongement
de $\kappa(\pi(x))$ dans $\hrt x^1$.

\begin{lemm} \label{redtemk}
Soit $\mathfrak X$ un schéma formel plat et séparé sur $k^\circ$
admettant un recouvrement localement fini par des ouverts affines
topologiquement de présentation finie sur $k^\circ$,
soit $x\in
\mathfrak X_{\eta}$ et soit $V$ l'adhérence de~$\pi(x)$ dans
$\mathfrak X_s$ (munie de sa structure réduite).  

\begin{enumerate}
\item Le $\tilde k^1$-schéma intègre 
$V$ est de type fini. 
\item 
L'ouvert quasi-compact $\widetilde{(\mathfrak X_\eta,x)}^1$ de $\zr {\hrt
x^1}{\tilde k^1}$ est égal à $\zr {\hrt x^1}{\tilde k^1}\{V\}$. 

\item Soit $\mathfrak U$ un ouvert formel de~$\mathfrak X$. Alors $\mathfrak U_\eta$
est un voisinage de~$x$ si et seulement si $\mathfrak U_s$ contient~$V$.

\item Le point $x$ est un point intérieur de~$\mathfrak X_{\eta}$ si et
seulement si $V$ est propre.
\end{enumerate}
\end{lemm}

\begin{proof}
Montrons (a). 
Soit $(\mathfrak U_i)_{i\in I}$ un recouvrement localement fini
de~$\mathfrak X$ par des ouverts affines
topologiquement de présentation finie sur~$k^\circ$.
Soit $J$ l'ensemble fini des $i\in I$ tels que $\pi(x)\in\mathfrak U_i$.
La réunion des $\mathfrak U_i$, pour $i\not\in J$, est
un ouvert de~$\mathfrak X$; son complémentaire est
donc fermé, il est contenu dans $\bigcup_{i\in J}\mathfrak U_i$,
et
contient~$\pi(x)$. Par
suite,
$V$ est un sous-schéma fermé
du $k$-schéma de type fini $\bigcup_{i\in J} \mathfrak U_{i,s}$, et il est donc de type fini. 

Comme $\pi^{-1}(V)$ est un
voisinage ouvert de~$x$ dans~$\mathfrak X_\eta$ qui est contenu dans
$\bigcup_{i\in J} \mathfrak U_{i,\eta}$, on peut remplacer $\mathfrak X$ par la
réunion des $\mathfrak U_i$ pour $i\in J$. Alors, $\widetilde{(\mathfrak
X_\eta,x)}=\bigcup_{i\in J} \widetilde{(\mathfrak U_{i,\eta},x)}$ et $V=\bigcup_{i\in J}
\mathfrak U_{i,s}\cap V$, et il suffit de montrer l'assertion requise
pour chacun des $\mathfrak U_i$, pour $i\in J$. En d'autres termes, 
pour prouver que
$\widetilde{(\mathfrak X_\eta,x)}^1$ coïncide 
avec $\zr {\hrt x^1}{\tilde k^1}\{V\}$,
on peut supposer que $\mathfrak X$ est un schéma formel affine
plat topologiquement de présentation finie sur~$k^\circ$. Écrivons
$\mathscr X=\Spf A$. 

La réduction stricte $\widetilde{(\mathfrak X_\eta,x)}^1$ est alors
égale à $\zr{\hrt x^1}{\widetilde k^1}\{R\}$, où $R$ désigne l'image 
du morphisme naturel de $\widetilde {A_k}^1$ dans 
$\hrt x^1$. Mais comme $\widetilde {A_k}^1$ est finie
sur $A\otimes_k \widetilde k^1$, l'anneau $R$ est fini
sur l'image $S$ du morphisme naturel de  $A\otimes_k \widetilde k^1$
dans $\hrt x^1$, si bien que 
 $\widetilde{(\mathfrak X_\eta,x)}^1$ est aussi égal à $\zr{\hrt x^1}{\widetilde k^1}\{S\}$. 
 Mais par construction, le morphisme de $A\otimes_k \widetilde k^^1$
dans $\hrt x_1$ est composé du morphisme d'évaluation de 
 $A\otimes_k \widetilde k^1$ vers $\kappa(\pi(x))$, dont l'image
 est l'anneau $\mathscr O(V)$, et du plongement canonique de 
 $\kappa(\pi(x))$ dans $\hrt x^1$ ; il en résulte que 
 $\widetilde{(\mathfrak X_\eta,x)}^1$ est égal à $\zr{\hrt x^1}{\widetilde k^1}\{V\}$, 
 ce qu'on souhaitait établir. 
 
Montrons (c). 
Soit maintenant $\mathfrak U$ un ouvert formel de~$\mathfrak X$.
Si $V\subset\mathfrak U_s$, alors $\pi^{-1}(V)$ est une partie ouverte
de~$\mathfrak X$ contenue dans $\mathfrak U_\eta$ et contenant~$x$,
de sorte $\mathfrak U_\eta$ est un voisinage de~$x$.
Supposons inversement que $\mathfrak U_\eta$ soit un voisinage de~$x$ ; 
(en particulier $x\in \mathfrak U_\eta$, si bien que 
$\pi(x)\in \mathfrak U_s$).
On a alors $\widetilde{(\mathfrak U_\eta,x)}^1=\widetilde{(\mathfrak X\eta,x)}_1$;
d'après (b), cela signifie que 
$\zr {\hrt x^1}{\tilde k^1}\{V\}= 
\zr {\hrt x^1}{\tilde k^1}\{V\cap\mathfrak U_s\}$.
Le morphisme dominant de $V\cap \mathfrak U_s$ vers $V$ est
donc propre par critère valuatif, ce qui entraîne que $V\cap \mathfrak U_s=V$, 
c'est-à-dire que $V\subset \mathfrak U_s$. 

Montrons enfin (d). Le point $x$ est un point intérieur
de $\mathfrak X_\eta$ si et seulement si $\widetilde{(\mathfrak X_\eta,x)}=
\zr{\hrt x^1}{\widetilde k^1}$, c'est-à-dire, au vu de (b), si et
seulement si $\zr{\hrt x^1}{\widetilde k^1}\{V\}=\zr{\hrt x^1}{\widetilde k^1}$. 
Par le critère valuatif de propreté, c'est le cas si et seulement si $V$ est propre
sur $\widetilde k^1$. 
\end{proof}

\section{Structures faiblement analytiques}

\begin{defi}
Soit $X$ un espace $k$-analytique, soit $A$ une partie fermée de~$X$
et soit~$d$ un entier naturel.
On dit que $A$ est faiblement analytique (de dimension~$\leq d$)
si pour tout domaine
analytique compact~$V$ de~$X$, il existe
une famille finie $(f_i\colon V_i\to W_i)$ de morphismes
d'espaces $k$-analytiques
définis sur des domaines analytiques compacts de~$V$
et une famille $(w_i)$, où $w_i\in W_i$ pour tout~$i$,
tels que $A\cap V=\bigcup_i f_i^{-1}(w_i)$
(et $\dim_{f_i}(x)\leq d$ pour tout $x\in A\cap V_i$).
\end{defi}

C'est une propriété G-locale.
Une union localement finie de parties faiblement analytiques
est naturellement munie d'une structure faiblement analytique.

\begin{exem}\label{exem.bord-fa}
Le bord d'un espace affinoïde~$X$ est
faiblement analytique
de dimension~$\leq \dim(X)-1$ : cela découle du théorème
3.2 de \cite{ducros2012b}. 
\end{exem}

\begin{rema}
Les notions de dimension d'une partie faiblement
analytique dépendent a priori des morphismes choisis
pour la présenter.
Par exemple, si $r\not\in\abs{k^\times}^{\Q}$,
l'application identique de~$\mathscr M(k_r)$
et le morphisme canonique $\mathscr M(k_r)\to\mathscr M(k)$
fournissent deux présentations de~$\mathscr M(k_r)$
comme partie faiblement analytique de lui-même,
de dimensions respectives~$0$ et~$1$. 
\end{rema}

\begin{prop}\label{prop.tube-faiblement-analytique}
Soit $X=\mathscr M(\mathscr A)$ un espace affinoïde,
soit $\widetilde X=\Spec(\widetilde{\mathscr A})$ sa réduction graduée
et soit $d$ un entier naturel. L'ensemble des points~$x\in X$
tel que $\dim(\overline{\pi(x)})\geq d$
est une partie faiblement analytique de dimension~$\leq \dim(X)-d$.
\end{prop}
En particulier, si $d=\dim(X)$, on retrouve en particulier que cet
ensemble est de dimension nulle ; on sait par ailleurs que c'est le bord de Shilov de~$X$.
\begin{proof}
Soit $V$ cette partie.
Choisisissons une famille finie $(f_1,\dots,f_n)$ d'éléments de~$\mathscr A$,
de normes spectrales $(r_1,\dots,r_n)$ non nulles,
telle que les $\widetilde{f_i}$ engendrent
l'algèbre graduée~$\widetilde{\mathscr A}$.
Soit $x\in X$ et soit $y=\pi(x)$.
Pour que $x\in V$, il faut et il suffit qu'il existe
une partie~$I\subset\{1,\dots,n\}$ de cardinal~$d$
telle que les $\widetilde{f_i}(y)$, pour $i\in I$, soient 
algébriquement indépendantes sur~$\tilde k$.
Cette dernière condition équivaut à l'égalité $f_I(x)=r_I$,
où $f_I=(f_i)_{i\in I}$ et $r_I=(r_i)_{i\in I}$.
D'autre part, pour $x\in f_I^{-1}(\eta_{r_I})$,
on a $\dim(X)\geq d_k(x) = d_{\hr{\eta_{r_I}}}(x)+d$,
de sorte que $d_{\hr{\eta_{r_I}}}(x)\leq \dim(X)-d$.
Comme $V$ est la réunion (finie) des $f_I^{-1}(\eta_{r_I})$,
cela prouve que c'est une partie faiblement analytique
de dimension~$\leq \dim(X)-d$.
\end{proof}

\chapter{Parties squelettiques d'un espace analytique}

\section{Tores et leurs squelettes canoniques} 

\subsection{Convention}
Dans la suite de ce travail, nous travaillerons très régulièrement avec 
les analytifiés $\A^{n,\an}$ et $\gm^{n,\an}$ des $k$-schémas
$\A^n$ et $\gm^n$, et seulement exceptionnellement avec ces schémas eux-mêmes. 
Aussi écrirons-nous désormais, afin d'alléger l'écriture, 
$\A^n$ et $\gm^n$ au lieu de $\A^{n,\an}$ et $\gm^{n,\an}$. Il sera donc considéré par défaut que
$\A^n$ et $\gm^n$ sont des espaces analytiques ; si nous avons besoin de travailler par endroit avec 
les schémas correspondants, nous le préciserons explicitement. 

\subsection{}
Dans ce travail, le groupe multiplicatif et ses puissances
jouent un rôle fondamental. Rappelons que l'anneau
des fonctions algébriques sur~$\gm^n$ est l'anneau $A=k[T_1^{\pm1},\dots,T_n^{\pm1}]$
des polynômes de Laurent à coefficients dans~$k$ 
en des indéterminées $T_1,\dots,T_n$ ;
l'espace analytique $\gm^n$ 
est l'ensemble des semi-normes
multiplicatives sur cet anneau qui étendent la valeur absolue de~$k$. 
L'espace $\gm^n$ est un ouvert de Zariski dense de $\A^n$ : c'est le lieu des points
en lesquels aucune fonction coordonnée ne s'annule. 
Pour tout polyrayon $r=(r_1,\ldots, r_n)$ le point
$\eta_r$ de de $\A^n$(défini à l'exemple \ref{exemple-etar})
appartient à $\gm^n$. Il correspond à la norme multiplicative
$\sum a_I T^I\mapsto \max \abs{a_I}r^I$ sur l'anneau des polynômes de Laurent. 
Si le corps $k$ a besoin d'être rappelé nous écrirons $\eta_{k,r}$ au lieu de $\eta_r$. 

Pour éviter le recours à un choix explicite de coordonnées,
nous adopterons la convention suivante : 
nous appellerons \emph{tore} (de dimension~$n$)
tout groupe analytique isomorphe à une
puissance~$\gm^n$ du groupe multiplicatif ;
nos tores sont donc implicitement supposés déployés. 

\subsection{}
Rappelons que toute fonction \emph{analytique} inversible sur~$\gm^n$ 
est un monôme, c'est-à-dire
de la forme $aT_1^{m_1}\dots T_n^{m_n}$,
où $a\in k^*$ et $(m_1,\dots,m_n)\in\Z^n$ est un vecteur appelé 
exposant.
(En dimension~1, cet énoncé est le lemme~9.7.1/1 
de~\cite{bosch-g-r1984},
le cas général en découle par récurrence.)
C'est en particulier une fonction algébrique.

Le morphisme $\gm^n\to \gm$ induit par une telle fonction est un homomorphisme
de groupes analytiques 
si et seulement si $a=1$. 

Soit $T$ un tore. L'ensemble $\mathsf X(T)$
des \emph{caractères} de~$T$, c'est-à-dire des morphismes 
de groupes analytiques de $T$ vers $\gm$, est un groupe abélien. 

On déduit de ce qui précède
que tout isomorphisme
de~$T$ avec~$\gm^n$ induit un isomorphisme de~$\mathsf X(T)$ 
avec~$\Z^n$ ; on en déduit aussi que si $T$ et~$T'$ sont deux tores, 
tout morphisme
d'espaces $k$-analytiques de~$T$ vers~$T'$ est affine, \emph{i.e.}
composé d'un morphisme de groupes analytiques et d'une translation ; 
c'est donc un morphisme de groupes si et seulement si il envoie
l'origine de~$T$ sur l'origine de~$T'$. Pour cette raison, lorsque
nous évoquerons un morphisme entre tores qui n'est pas nécessairement
un morphisme de groupes analytiques, nous le qualifierons de morphisme affine.

\subsection{Tropicalisations}
On considère le groupe abélien~$\R_+^\times$ comme un $\R$-espace 
vectoriel, \emph{via} l'exponentiation. 
Par ailleurs, le choix
d'une base de logarithmes permet de l'identifier à 
l'espace~$\R$. On fixe une fois pour toutes  un tel isomorphisme
\mbox{$\log\colon\R_+^\times\ra\R$}\index{log@$\log$}. 
Dans ce texte, il n'y a pas d'intérêt
particulier à choisir le logarithme néperien; en particulier,
nous n'aurons jamais à dériver cet isomorphisme, vu comme fonction
d'une variable réelle.
L'utilisation de cet isomorphisme $\log$ présentera pour nous deux intérêts : elle
nous permettra d'utiliser la géométrie paralinéaire sous sa forme standard, plutôt
que sa variante multiplicative ; et elle donnera lieu à un formalisme apparenté à celui 
de la géométrie complexe, dans lequel les logarithmes de fonctions holomorphes
inversibles jouent un rôle majeur.

Soit $T$ un tore. 
Nous noterons $T_\trop $ le $\R$-espace vectoriel $\Hom(\mathsf 
X(T),\R^\times_+)$, 
que l'on peut identifier à $(\R^\times_+)^n$ 
une fois choisie une trivialisation de $T$, 
ou à~$\R^n$ via l'isomorphisme~$\log$ choisi.
L'espace vectoriel~$T_\trop$ est appelé \emph{tropicalisation} du 
tore~$T$.\index{tropicalisation!d'un tore}

Étant donné un point~$t$ de~$T$, l'application qui applique
un caractère~$\phi\in\mathsf X(T)$ sur l'élément $\abs{\phi(t)}$
est un élément de~$T_\trop$ noté $\trop(t)$. L'application
$\trop\colon T\ra T_\trop$ ainsi définie est continue.

La formation de $T_\trop$ est fonctorielle : tout morphisme affine 
$f\colon T\to T'$ entre tores induit une application affine $f_\trop 
\colon T_\trop\to T'_\trop$ telle que  le diagramme
\[ \begin{tikzcd} T\ar{d} \ar{r}{f} &T'\ar{d} \\ T_\trop \ar{r}{f_\trop} & 
T'_\trop\end{tikzcd}\]
 commute. 

\subsection{Squelette de $\gm^n$}\label{ss.skel-gmn}
L'application de tropicalisation~$\gm^n \to (\gm^n)_\trop$ 
possède une section naturelle, notée~$\sq$
et définie comme suit : 
pour tout point~$r=(r_1,\ldots, r_n)\in (\gm^n)_\trop=\R^n$,
on note $\sq(r)$ le point $\eta_r$ de $\gm^n$. 
L'application~$\sq$ est continue;
elle induit un homéomorphisme de~$(\R_+^\times)^n$ sur un sous-espace
fermé~$S(\gm^n)$ de~$\gm^n$
que l'on appelle 
\emph{squelette canonique}\index{squelette canonique! de $\gm^n$} de~$\gm^n$.

On munit~$S(T)$ de l'unique structure de groupes pour laquelle~$\sq$
est un isomomorphisme de groupes.
Pour tout extension valuée~$K$ de~$k$,
l'application naturelle $T(K)\to S(T)$ donnée par $t\mapsto \sq(\trop(t))$
est un morphisme de groupes.

Par définition, pour qu'un point $x\in\gm^n$ appartienne à~$S(\gm^n)$,
il faut et il suffit que pour tout polynôme de Laurent $f=\sum a_m T^m$,
on ait
\[ \abs{f(x)} = \max_m \abs{a_m}\cdot  \abs{T_1(x)}^{m_1}\dots \abs{T_n(x)}^{m_n}. \]
(Si $x=\sq(r)$, la formule est vraie par définition; dans l'autre sens, 
cette formule prouve que $x=\eta_r$ avec $r_i=\abs{T_i(x)}$ pour tout~$i$.)

Du point de vue de la réduction graduée, cette dernière condition
revient à demander que les éléments $\widetilde{T_i(x)}$
du corps résiduel \emph{gradué} $\hrt x$ soient algébriquement indépendants
sur~$\tilde k$
(\ref{ss-independance-alg-residuelle}). 

\begin{lemm}\label{lemm.squelette-canonique}
\begin{enumerate}\def\theenumi{\alph{enumi}}\def\labelenumi{\theenumi)}
\item
Le squelette canonique $S(\gm^n)$ de~$\gm^n$ est 
l'ensemble des points de $\gm^n$ qui sont maximaux
pour la relation $\leq$ définie comme suit : $x\leq y$ si et seulement si 
$\abs{f(x)}\leq \abs{f(y)}$ pour toute fonction analytique~$f$ 
sur~$\gm^n$. 
\item
Soit $t\in S(\gm^n)$. Les réductions graduées $\widetilde{T_i(t)}$
sont algébriquement indépendantes et engendrent $\hrt t$.  
En particulier, on a $d_k(t) =n$.
\item
Pour tout ouvert de Zariski non vide~$U$ de~$\gm^n$,
on a $S(\gm^n)\subset U^\an$.
\end{enumerate}
\end{lemm}
\begin{proof}
\emph a)
Soit~$x\in \gm^n$. Pour tout $i$, 
on pose~$r_i=T_i(x)$, et l'on note $r$ le $n$-uplet $(r_1,\ldots, r_n)$. 
Soit~$f=\sum_{m\in\Z^n} a_m T_1^{m_1}\dots T_n^{m_n}$
une fonction analytique sur $\gm^n$. On a 
\[ \abs{f(x)}\leq\max _{m\in\Z^n} \abs{a_m} r_1^{m_1}\dots r_n^{m_n}=\abs{f(\eta_r)}.\]
En conséquence, $x\leq \eta_r$ ; en particulier, si $x$ est maximal alors $x=\eta_r$. 
Ainsi, tout point maximal appartient à $S(\gm^n)$. 

Réciproquement, 
soit $r\in (\R^\times_+)^n$ et soit~$x\in\gm^n$ tel que $x\geq \eta_r$. 
Pour tout $i$, on pose 
$s_i=\abs{T_i(x)}$, et l'on note $s$ le $n$-uplet $(s_1,\ldots, s_n)$. 
D'après ce qui précède, on a $x\leq \eta_s$ et,  \emph{a fortiori}, $\eta_r\leq \eta_s$. En appliquant cette inégalité à 
$T_i$ et $T_i^{-1}$ pour tout $i$, il vient $r=s$. Ainsi,
$\eta_r\leq x\leq \eta_r$
et $x=\eta_r$. Par conséquent, $\eta_r$ est maximal. 

\emph b)
C'est une conséquence de l'exemple \ref{exem-ktr-tilde} et de la densité de $k(T_1,\ldots, T_n)$
dans $\hr {\eta_r}$. 


\emph c)
Comme $\gm^n$ est irréductible algébriquement, donc analytiquement,
tout fermé de Zariski strict de~$\gm^n$ est de dimension~$<n$.
Ainsi, tout point de~$S(\gm^n)$ est Zariski dense.
L'assertion en résulte.
\end{proof}

\begin{coro}\label{coro.skel-gmn-funct}
Soit $f\colon \gm^n\to\gm^n$ un $k$-morphisme fini et plat,
de degré~$d$. Alors $f^{-1}(S(\gm^n))=S(\gm^n)$
et $f$ induit un homéomorphisme de~$S(\gm^n)$ sur lui-même. 
De plus, pour tout $x\in S(\gm^n)$, $\hrt x$ est degré~$d$
sur~$\hrt {f(x)}$ et $f^{-1}(f(x))=\mathscr M(\hr x)$,
en particulier, cette fibre est réduite.
\end{coro}
\begin{proof}
Le morphisme~$f$ est fourni par des monômes $c_{1} T^{m_1}, \dots, c_n T^{m_n}$,
où les~$m_i$ sont des multi-indices.
La matrice $\mathrm A=(m_1,\dots,m_n)$ est à coefficients entiers ;
on a $\abs{\det(A)}=d$. En particulier, cette matrice
est inversible dans~$\GL(n,\Q)$.
On en déduit que pour tout $x\in\gm^n$,
les $\widetilde{T_i(x)}$ sont algébriquement indépendents
si et seulement si il en est de même des $\widetilde{c_i} \widetilde{T(x)}^{m_i}$; ce qui prouve que $x\in S(\gm^n)$ si et seulement si $f(x)\in S(\gm^n)$.
Si c'est le cas, le corps engendré par les $\widetilde{T_i(x)}$
est de degré~$d$ sur celui engendré par les $\widetilde{c_i} \widetilde{T(x)}^{m_i}$. 
L'extension $[\hr x:\hr{f(x)}]$ est donc de degré au moins~$d$.
Comme $f$ est fini et plat de degré~$d$,
il en résulte que $x$ est le seul point de la fibre~$f^{-1}(f(x))$
et que cette fibre est réduite en~$x$. On a donc $f^{-1}(f(x))=\mathscr M(\hr x)$.
\end{proof}

\begin{coro}\label{coro-skel-ginvariant}
\begin{enumerate}
\item
Soit $\sigma\colon k\to K$ un morphisme isométrique de corps ultramétriques complets. 
Le morphisme $p$ d'espaces analytiques $\gm[K]^n\to \gm[k]^n$ qu'il induit
applique~$\eta_{K,r}$ sur~$\eta_{k,r}$ pour tout~$r$.

\item Supposons que $\widetilde K$ est algébrique sur $\sigma(\widetilde k)$. Le point
$\eta_{K,r}$ est alors l'unique antécédent de $\eta_{k,r}$ par $p$ et le morphisme naturel 
de $\hr{\eta_{k,r}}\widehat\otimes_k K$ vers $\hr{\eta_{K,r}}$ est un isomorphisme d'anneaux normés.

\item Pour tout~$x\in\gm^n$, le corps~$k$ est séparablement clos dans~$\hr x$. 
\end{enumerate}
\end{coro}
\begin{proof}
\begin{enumerate}
\item
Pour tout polynôme de Laurent $f=\sum a_m T^m$ appartenant à $k[T,T^{-1}]$ on a 
\[\abs{f(p(\eta_{K,r})}=\abs{\sigma(f)(\eta_{K,r})}=\sup_m \abs{\sigma(a_m)}r^m=\sup_m \abs{a_m}r^m=\abs{f(\eta_{k,r})},\]
ce qui conclut. 
 
\item Pour alléger les notations, utilisons $\sigma$ pour identifier $k$ à un sous-corps de $K$. 
Soit $y$ un antécédent de $\eta_{k,r}$ sur $\eta_{K,r}$. On a $\abs{T_i(y)}=
\abs{T_i(x)}=r_i$ pour tout $i$, et comme 
$\widetilde K$ est algébrique sur $\widetilde k$, les $\widetilde{T_i(y)}=\widetilde{T_i(x)}$ sont encore 
algébriquement indépendants sur $\widetilde K$ ; par conséquent, $y=\eta_{K,r}$. 
Ceci entraîne que $\mathscr M(\hr{\eta_{k,r}}\widehat\otimes_k K)=\{\eta_{K,r}\}$. 
Par ailleurs, la norme de $\hr{\eta_{k,r}}\widehat\otimes_k K$ est multiplicative 
(\cf \cite{ducros2009}, lemme 1.8 et exemples 1.9). C'est donc elle qui définit l'unique point 
$\eta_{K,r}$ de $\mathscr M(\hr{\eta_{k,r}}\widehat\otimes_k K)$. Ainsi $\hr{\eta_{k,r}}\widehat\otimes_k K$
s'injecte isométriquement dans $\hr{\eta_{K,r}}$. 
Il en résulte que si $\lambda$ est un 
un élément non nul de $\hr{\eta_{k,r}}\widehat\otimes_k K$ alors $\lambda$ ne s'annule pas
sur $\mathscr M(\hr{\eta_{k,r}}\widehat\otimes_k K)=\{\eta_{K,r}\}$, ce qui entraîne 
que $\lambda$ est inversible (\cite{berkovich1990}, Corollary 2.1.4). L'injection 
isométrique de $\hr{\eta_{k,r}}\widehat\otimes_k K$
dans $\hr{\eta_{K,r}}$ a donc pour image un sous-corps complet de ce dernier. Comme le corps
des fractions de cette image est dense dans le corps résiduel complété $\hr{\eta_{K,r}}$, 
l'injection isométrique de $\hr{\eta_{k,r}}\widehat\otimes_k K$
dans $\hr{\eta_{K,r}}$ est finalement bijective. 

\item
Soit $K$ le complété d'une clôture algébrique de~$k$.
Il s'agit de démontrer que tout point~$x\in S(\gm^n)$
possède un unique antécédent dans~$\gm[K]^n$.
On peut le déduire de \emph a) qui assure que 
pour tout  $r\in(\R_+^\times)^n$
le point~$\eta_{K,r}$ est 
invariant par
l'action de~$\Aut_k(K)$
et s'applique sur le point~$\eta_{k,r}$ de~$\gm^n$, dont il est dès lors 
l'unique antécédent. On peut aussi le déduire de \emph b) puisque $\widetilde K$ est algébrique 
sur $\widetilde k$. 
\qedhere
\end{enumerate}
\end{proof}

\subsection{Squelette canonique d'un tore} 
Soit~$T$ un tore et soit
$f\colon \gm^n\to T$ un isomorphisme. En vertu du corollaire
précédent,
le sous-espace fermé~$f(S(\gm^n))$ de $T$ ne dépend pas de~$f$. 
On l'appelle le \emph{squelette canonique}\index{squelette canonique! d'un tore} de~$T$
et on le note~$S(T)$. 

L'application naturelle $T\to T_\trop$
induit un homéomorphisme $S(T)\simeq T_\trop$.
Elle munit~$S(T)$
d'une structure naturelle d'espace vectoriel réel de dimension~$\dim(T)$.

Notons $\sq\colon T_\trop \to T$ l'homéomorphisme réciproque.
Si $P$ est un sous-ensemble de~$T_\trop$, on notera aussi~$P_\sq$ 
le sous-ensemble~$\sq(P)$ de~$S(T)$
qui lui correspond \emph{via} cet homéomorphisme. 

Le corollaire~\ref{coro-skel-ginvariant} s'étend
immédiatement au cas des tores.

Soit $t$ un point de $T$ et soit $\chi_1,\ldots, \chi_n$ une base
du groupe des caractères de $T$. Le degré de transcendance
sur $\widetilde k$ du sous-corps gradué
de $\hrt t$ engendré par les $\widetilde k$ et les $\widetilde{\chi_i(t)}$
ne dépend pas de la base $(\chi_i)$ : on le déduit d'une fait qu'on passe d'une
base de $\mathsf X(T)$ à une autre par une transformation monomiale inversible à
exposants entiers. Ce degré sera appelé le degré de transcendance résiduel 
(gradué) torique de $t$. 
Le point~$t$ appartient à~$S(T)$ si et seulement si son degré de 
transcendance résiduel torique est égal à $n$. 

\subsection{Sections de Shilov}\label{section.shilov}
Soit~$T$ un tore
et soit~$X$ un espace~$k$-analytique. 
Pour tout~$x\in X$, notons $\sq_x$ l'application
de~$T_\trop$ dans~$T\times_k X$ qui applique
un point $t\in T_\trop$ sur le point qui correspond 
à~$t$ sur le squelette canonique du~$\hr x$-tore~$T_{\hr x}$,
fibre au-dessus de~$x$ de la projection de $T\times_k X$ sur~$X$.
Pour tout $x\in X$, l'application $\sq_x$ est continue.

\begin{lemm}\label{lemm-continuite-shilov}
L'application 
$T_\trop\times X\ra T\times_k X$, $(t,x)\mapsto \sq_x(t)$
est continue.
\end{lemm}

En particulier, 
pour tout~$t\in T_\trop$,
l'application~$x\mapsto \sq_x(t)$ de~$X$ vers~$T\times_k X$
est une section \emph{continue} (parfois dite \emph{de Shilov})
de la seconde projection (\cite{berkovich1990}, lemma 3.2.2, (i)). 

\begin{proof}
La preuve est en fait la même que celle de la continuité de~$\sq_x$,
pour $x\in X$ fixé.
On peut supposer que $X=\mathscr M(A)$ est affinoïde
et que $T=\gm^n$; alors,  $T\times_k X$ est l'analytifié
du schéma $\gm[\mathscr A]^n$ et sa topologie
est définie par les normes des \emph{polynômes}  de Laurent
à coefficients dans~$\mathscr A$.  
La continuité de l'application étudiée découle alors de ce
que  pour tout $f=\sum a_I T^I\in A[T_1^{\pm1},\dots,T_n^{\pm1}]$
et toute partie ouverte~$J$ de~$\R_+$, 
l'ensemble des couples $(r, x) \in \R^n\times X$ 
tels que $\sup(\abs{a_I(x)} r^I) \in J$ est ouvert.
\end{proof}

\subsection{}\label{sq_X(P)}
Pour toute partie~$P$ de~$T_\trop$, on note $\sq_X(P)$
l'ensemble des points de~$T\times_k X$
de la forme $\sq_x(t)$, pour $t\in P$ et $x\in X$.

Supposons que $X$ soit connexe (il est alors connexe par arcs) et 
soit~$P$ une partie connexe (par arcs) de~$T_\trop$ ; démontrons que
l'ensemble $\sq_X(P) $ est connexe (par arcs). Donnons-nous deux
points~$t,s$ de~$P$ et deux points~$x,y$ de~$X$.
Comme~$P$ est connexe par arcs et comme $u\mapsto \sq_x(u)$
est  continue, les points~$\sq_x(t)$ et $\sq_x(s)$ sont 
 dans la même composante connexe (par arcs) de~$\sq_X(P)$;
comme la section de Shilov $\xi\mapsto s_{\sq,\xi}$ est continue,
les points~$\sq_x(s)$ et $\sq_y(s)$ 
appartiennent à la même composante connexe (par arcs) de~$\sq_X(P)$.
Par suite, $\sq_X(P)$ est connexe (par arcs).

\subsection{Une convention}\label{ss.convention}
Soit $T$ et $T'$ deux tores. 
Les $\R$-espaces vectoriels 
$T_\trop$ et $T'_\trop$ sont respectivement égaux à
$\Hom(\mathsf X(T),\R^\times_+)$ 
et $\Hom(\mathsf X(T'),\R^\times_+)$. 
On dispose par conséquent d'un isomorphisme naturel 
\[\Hom_{\R-\text{lin}}(T_\trop, T'_\trop)\simeq\Hom(\mathsf X(T'), \mathsf X(T))\otimes_{\Z}\R.\]
Dans le cas particulier où $T'=\gm$, on a $T'_\trop=\R^\times_+$
d'où un isomorphisme
naturel
 \[\Hom_{\R-\text{lin}}(T_\trop, \R^\times_+)\simeq \mathsf X(T)\otimes_{\Z}\R.\]
En particulier, l'espace vectoriel~$T_\trop$ 
possède une $\Q$-structure canonique,
donnant lieu aux notions de cellules, polyèdres, décomposition cellulaires,
$\Q$-rationnelles. 
De même, on dit qu'une application affine $f\colon T_\trop\ra T'_\trop$
est $\Q$-rationnelle si sa partie linéaire appartient 
à $\Hom(\mathsf X(T'),\mathsf X(T))\otimes\Q$.

Désormais, et sauf mention expresse du contraire,
nous dirons cellule, polyèdre, décomposition cellulaire 
pour cellule $\Q$-rationnelle, polyèdre $\Q$-rationnel,
décomposition cellulaire $\Q$-rationnelle de~$T_\trop$.
De même, nous dirons simplement qu'une application de $T_\trop$ dans~$T'_\trop$
est affine si elle est affine et $\Q$-rationnelle,
et qu'une application continue $f$ d'un polyèdre~$P$ de~$T_\trop$
vers~$T'_\trop$ est paralinéaire 
s'il existe une décomposition cellulaire $\Q$-rationnelle ~$\mathscr C$ de~$P$ 
telle que la restriction de~$f$ à chaque cellule de $\mathscr C$
s'étende en une application affine $\Q$-rationnelle de~$T_\trop$ 
vers~$T'_\trop$. 
 
%
%
%

Compte tenu des isomorphismes canoniques
de $S(T)$ sur~$T_\trop$  et de $S(T')$ sur~$T'_\trop$,
la même convention vaudra pour les cellules, polyèdres, applications affines, 
etc. de~$S(T)$.

%

\subsection{Caractérisation de l'image réciproque du squelette canonique par des formules
explicites}\label{imrecisquel-explicite}
Soit~$X$ un espace~$k$-analytique et
soit~$f=(f_1,\ldots,f_n)$ une famille de fonctions analytiques
inversibles sur~$X$ ; on note~$\phi$ le morphisme~$X\to \gm^n$
défini par les~$f_i$. 

Soit~$x\in X$ ; pour tout~$i$,
posons~$r_i=\abs{f_i(x)}$. Le point~$\phi(x)$ 
appartient à~$S(\gm^n)$ si et seulement
si il est égal à~$\eta_{r}$ ; par définition, 
cela revient à demander que 
\[\left|\sum a_I f(x)^I\right|=\max |a_I|\cdot r^I\]
pour tout polynôme~$\sum a_I T^I\in k[T]$. 
Compte tenu de la description du squelette de~$\gm^n$ 
(\S\ref{ss.skel-gmn}),
cela revient aussi à dire que les $\widetilde{f_i(x)}$
sont algébriquement indépendants sur~$\tilde k$.

Ce critère entraîne en particulier les résultats suivants
qui nous seront utiles par la suite.

1) 
Soit~$\alpha_1,\ldots,\alpha_n$ des éléments de~$k^\times$, 
et soit~$\psi \colon X\to \gm^n$ le morphisme induit par les~$\alpha_i f_i$. 
On a $\psi^{-1}(S(\gm^n))=\phi^{-1}(S(\gm^n))$.


2) 
Soit~$(g_1,\ldots, g_n)$ une famille
de fonctions analytiques sur~$X$. Supposons que~$\abs{g_i(x)}<r_i$
pour tout~$i$, et 
soit~$\chi$ le morphisme de but~$\gm^n$ 
défini par les fonctions~$f_i+g_i$. 

Alors $\phi(x)\in S(\gm^n)$ si et seulement si $\chi(x)\in S(\gm^n)$.

 
3) Soit~$\alpha_1,\ldots, \alpha_n$ des éléments de~$k^\times$, 
et~$u_1,\ldots, u_n$ des fonctions analytiques sur~$X$, 
telles que~$\abs {u_i(y)}<\abs {\alpha_i}$ pour tout~$i$
et tout~$y$
appartenant à~$X$. 
Soit~$\theta \colon  X\to \gm^n$ le morphisme
induit par les fonctions~$(\alpha_i+u_i)f_i$. 
Posons~$s_i=|\alpha_i|r_i$. On a alors
les équivalences
\[\phi(x)\in S(\gm^n)\iff \phi(x)=\eta_{r}\iff \theta(x)=\eta_{s}\iff \theta(x)\in S(\gm^n).\]

%

\section{Moments, tropicalisations, polyèdres}

\begin{defi}
Soit $X$ un espace $k$-analytique  ; 
on appelle \emph{moment} sur~$X$ un morphisme de~$X$ vers un tore.
\end{defi}
Soit $f\colon X\to T$ un moment sur un espace $k$-analytique~$X$.
On désignera par $f_\trop\colon X\ra T_\trop$
l'application composée $X\to T\to T_\trop$.
On dira que l'application~$f_\trop$ est la \emph{tropicalisation} 
de~$f$ ;
c'est une application 
continue.\index{moment}\index{tropicalisation!d'un moment}

\begin{rema} 
Si $f\colon T\to T'$ est un morphisme affine entre tores, la
notation $f_\trop$ désigne donc aussi bien l'application induite 
$T_\trop \to T'_\trop$ que sa composée $T\to T'\to T'_\trop$ avec
l'application de tropicalisation de~$T$.
Cette ambiguïté n'aura, en pratique, guère de conséquences.
\end{rema}

\begin{lemm}\label{lemm.dense-moment}
Soit $X$ un espace analytique, soit $x$ un point de~$X$
et soit $V$ un domaine analytique de~$X$ contenant~$x$.
Soit $f\colon V\to \gm$ une fonction analytique inversible.
On suppose que $X$ est presque bon en~$x$.
Il existe alors un voisinage ouvert~$U$ de~$x$ et une fonction
analytique inversible~$g$ sur~$U$ telle que 
$\abs{f-g}<\abs{f}$ sur $U\cap V$; en particulier, 
$f_\trop$ et $g_\trop$ coïncident sur~$U\cap V$.
\end{lemm}
\begin{proof}
Soit $U$ un voisinage ouvert de~$x$ dans~$X$ et $g\in\mathscr O_X(U)$
tels que $\abs{f(x)-g(x)}<\abs{f(x)}$; il en existe car $\kappa(x)$ est
dense dans~$\hr x$ et $\abs{f(x)}>0$.
L'ensemble des points de~$U\cap V$ en lesquels $\abs{f-g}<\abs{f}$
est un voisinage de~$x$ dans~$V$; quitte à restreindre~$U$,
on peut supposer que cette inégalité vaut dans~$U\cap V$ tout entier.
Quitte à restreindre encore~$U$, on peut supposer que $g$ est inversible sur~$U$, d'où le lemme.
\end{proof}

\begin{theo}\label{prop.trop.global}
Soit $X$ un espace $k$-analytique compact et soit $f\colon X\to T$ 
un moment.
Le sous-espace compact $f_\trop(X)$ est un polyèdre de $T_\trop$ 
de dimension~$\dim(X)$.
\end{theo}
L'assertion
a été établie par Berkovich dans le cas où $X$ est G-localement
algébrisable, \emph{cf.} \cite{berkovich2004}, cor. 6.2.2. 
Pour le cas général, 
on pourra se reporter à \cite{ducros2003b}, \S 3.30,
ou à \cite{ducros2012b}, th.~3.2.

\begin{coro}\label{coro.dim-ftrop}
Soit $X$ un espace $k$-analytique compact et soit $V$ une partie fermée de~$X$
munie d'une structure faiblement analytique de dimension~$d$.
Soit $f\colon X\to T$ un moment.
Alors, $f_\trop(V)$ est une partie paralinéaire de dimension~$\leq d$ de~$T_\trop$.
\end{coro}
\begin{proof}
On se ramène immédiatement au cas où il existe un morphisme $g\colon X\to Y$
et un point $y\in Y$ tel que $V=g^{-1}(y)$.
Alors, le produit fibré $X\times_Y \hr y$ permet de considérer~$V$ comme un espace $\hr y$-analytique compact, de dimension~$\leq d$,
et $f|_V\colon V\to T_{\hr y}$ comme un moment sur cet espace.
On donc $\dim (f_\trop(V))\leq d$.
\end{proof}

\begin{coro}\label{coro.dim-ftrop-bord}
Soit $X$ un espace $k$-analytique compact de dimension~$n$.
Soit $f\colon X\to T$ un moment.
Si $X$ est affinoïde, alors $f_\trop(\partial (X))$ 
est une partie paralinéaire de dimension~$\leq n-1$ de~$T_\trop$;
en genéral, $f_\trop(\partial(X))$ est contenu dans 
une telle partie paralinéaire.
\end{coro}
\begin{proof}
Compte tenu de l'exemple~\ref{exem.bord-fa}
selon lequel $\partial(X)$ est une partie faiblement analytique
de dimension~$\leq \dim(X)-1$,
cela résulte du corollaire précédent.
Il s'agit en fait du théorème~3.2 de \cite{ducros2012b},
démontré de la même façon, même si la notion de structure
faiblement analytique n'avait pas été dégagée.
\end{proof}

\subsection{}Soit $X$ un espace
$k$-analytique et $f\colon X\to T$ 
un moment. Soit $x\in X$. 
On appellera degré de transcendance résiduel (gradué) 
de $f$ en $x$ le degré de transcendance résiduel torique du point 
$f(x)$ de $T$. Si $\chi_1,\ldots, \chi_n$ est une base de $\mathsf X(T)$, ce degré est
le degré de transcendance sur $\widetilde k$ 
du sous-corps gradué de $\hrt x$ engendré par
$\widetilde k$ et les $\widetilde{\chi_i\circ f(x)}$. 
Si $T=\gm^n$ et $f=(f_1,\ldots, f_n)$, c'est donc 
le degré de transcendance sur $\widetilde k$ 
du sous-corps gradué de $\hrt x$ engendré par
$\widetilde k$ et les $\widetilde{f_i(x)}$. 

Le degré de transcendance résiduel de $f$ en $x$ est majoré à la fois par $d_k(x)$ et par $n$. Il vaut $n$ 
si et seulement si $f(x)\in S(T)$. 

\begin{prop}
\label{prop.trop.local}
Soit $X$ un espace $k$-analytique et soit $x$ un point de $X$. 
Soit  $f=X\to T$ un moment ; posons 
$\xi=f_\trop(x)$. Notons $d$ le degré de transcendance résiduel de $f$
en $x$.

Il existe un voisinage analytique compact~$U$ de~$x$ dans~$X$ 
tel que pour pour tout voisinage analytique 
compact~$V$ de $x$ contenu dans~$U$, on ait les propriétés suivantes:
\begin{enumerate}\def\theenumi{\roman{enumi}}
\item Les germes de polyèdres $(f_\trop(U), \xi)$ 
et $(f_\trop(V),\xi)$ coïncident;

\item La dimension locale $\dim_\xi(f_\trop(V))$
est majorée par $d$, et a fortiori par $d_k(x)$. 

\item
Si $x\not\in\partial(X)$, alors $f_\trop(V)$ est purement
de dimension~$d$ en~$\xi$.
Si, de plus, $d=n$, alors 
$f_\trop(V)$  est un voisinage de~$\xi$ dans~$T_\trop$.

\item
Si $\dim_\xi(f_\trop(V))=n$, alors $f(x)\in S(T)$ ;
inversement, 
si $f(x)\in S(T)$ et si de plus $x\notin \partial (X)$, alors $\dim_\xi(f_\trop(V))=n$.
\end{enumerate}

\end{prop}
\begin{proof}
L'existence d'un voisinage analytique~$U$ vérifiant les i), ii) et iii) est le théorème~3.4 de~\cite{ducros2012b}.
Montrons (iv). 
Si $\dim_\xi(f_\trop(V))=n$, il résulte de (ii) que $n\leq d$, donc que $d=n$ et que $f(x)\in S(T)$. 
Réciproquement, si $f(x)\in S(T)$ alors $d=n$ et si de plus $x\notin \partial X$ il résulte
de (iii) que $\dim_\xi(f_\trop(V))=n$. 
\end{proof}

\begin{rema}\label{rem-description-trop-germ}
Soit $X$ un espace $k$-analytique et soit
$x\in X$. Soit $f\colon X\to \gm^n$ un moment. 
Si $V$ est un voisinage analytique compact de 
$x$ dans $X$, il résulte de la proposition ci-dessus
que le germe de $f_\trop(V)$ en $x$ ne dépend pas de $V$. 
Nous le noterons $f_\trop(X,x)$. 

La preuve du théorème 3.4 de \cite{ducros2012b} en donne une description 
en termes de réduction de Temkin, que nous allons rappeler ici, en supposant
que le germe est~$(X,x)$ séparé. 

Le théorème~\ref{theo.foncteurs-polyedraux} assure (en reprenant le vocabulaire
et les notations de~\ref{ss.vocabulaire-foncteur-polyedral})
l'existence d'un sous-foncteur $D$
de $\mathsf D_{\{1\}}^n$ tel que pour tout $H\in \mathsf D$ et tout
$n$-uplet $(h_1,\ldots, h_n)$ d'éléments de $H$, les assertions suivantes soient
équivalentes : 
\begin{itemize}
\item[(i)] $h\in D(H)$ ; 
\item[(ii)] il existe $H'\in \mathsf D_H$ et une valuation $w$
appartenant à $\widetilde{(X,x)}$ à valeurs dans $H'$ telle que 
$w(\widetilde{f_i(x)})=h_i$ pour tout $i$. 
\end{itemize}
Comme le groupe des paramètres est ici $\{1\}$, le foncteur $D$
est simplement défini par conjonctions et disjonctions d'inégalités
de la forme $\phi\leq 1$ où $\phi$ est un monôme $\prod x_i^{r_i}$ 
à exposants rationnels. L'ensemble $D(\R_+^\times)$
est alors un cône
rationnel, et
$f_\trop(X,x)=(f_\trop(x)\cdot D(\R_+^\times),f_\trop(x))$
en notation
multiplicative. 

En particulier, si $x$ est intérieur, 
c'est-à-dire si $\widetilde{(X,x)}=\zr {\hrt x}{\widetilde k}$,
ce germe de polytope ne dépend que 
de la classe d'isomorphie de l'extension graduée~$F$ 
de $\widetilde k$ dans~$\hrt x$ engendrée par les $\widetilde {f_i(x)}$,
munie de cette famille.
Dans ce cas, il est équidimensionnel, de dimension égale au degré de transcendance (gradué) de~$F$ sur~$\widetilde k$.
\end{rema}

\begin{defi}
Soit $X$ un espace analytique et soit $f\colon X \to T$ un moment.
Soit $x$ un point de~$X$.
On appelle dimension tropicale de~$f$ en~$x$  la borne inférieure
de dimensions des polyèdres~$f_\trop(V)$, où $V$ est un domaine analytique
compact qui contient~$x$. On la note $\dim_{\trop,x}(f)$.
\end{defi}

Soit $d$ le degré de transcendance résiduel de $f$ en $x$. 
D'après le théorème~\ref{prop.trop.global} et
la proposition~\ref{prop.trop.local},
on a
\begin{equation}\label{eq.dimtrop-dk-dim}
\dim_{\trop,x}(f) \leq d\leq d_k(x) \leq \dim_x(X),
\end{equation}
et $\dim_{\trop,x}(f)=d$ si $x\notin \partial X$. 
Si $X$ est compact, il résulte de la définition que
\begin{equation}\label{eq.dimtrop-sup}
 \dim(f_\trop(X)) = \sup_{x\in X} \dim_{\trop,x}(f).
\end{equation}

\begin{prop}\label{prop.dim-trop-max}
Soit $X$ un espace analytique, soit $x$ un point de~$X\setminus \partial (X)$.
Il existe un voisinage analytique compact~$V$ de~$x$ dans~$X$
et un moment $f$ sur~$V$ dont la dimension tropicale en~$x$
égale $d_k(x)$.
\end{prop}

\begin{proof}
Posons $d=d_k(x)$ et soit $a_1,\ldots, a_d$
une base de transcendance de $\hrt x$ sur
$\widetilde k$. Comme $x\notin \partial X$, l'espace $X$ est bon, et a fortiori
presque bon, en $x$. Il existe donc un voisinage analytique compact
$V$ de $x$ dans $X$ et $d$ fonctions analytiques inversibles $f_1,\ldots, f_d$
sur $V$ telles que $\widetilde{f_i(x)}=a_i$ pour tout $i$. Le degré de transcendance
résiduel de $f$ en $x$ est alors égal à $d$ par construction. 
Puisque $V$ est un voisinage de $x$ dans $X$, le point $x$ n'appartient pas
à $\partial V$ ; il s'ensuit que la dimension tropicale de $f$ en $x$ 
est elle aussi égale à $d$. 
\end{proof}

\begin{prop}\label{prop.abs-f-constante}
Supposons le corps~$k$ algébriquement clos.
Soit $X$ un espace $k$-analytique et soit $f$ une fonction analytique 
inversible sur~$X$.
Soit $x$ un point de~$X\setminus\partial(X)$
telle que $\abs f$ soit constante au voisinage de~$x$.
Il existe un élément $a\in k^\times$ et une fonction
$g\in \mathscr O_{X,x}^\times$ tels que $f=a+g$ et $\abs g<\abs a$
au voisinage de~$X$.
\end{prop}

\begin{proof}
L'hypothèse signifie que $f$ est de dimension tropicale nulle
en $x$. Puisque $x\notin \partial X$, la dimension tropicale de $f$ 
en $x$ coïncide avec son degré de transcendance résiduel. 
Celui-ci est donc nul, ce qui veut dire que $\widetilde{f(x)}$ est algébrique
sur $\widetilde k$, c'est-à-dire encore que $\widetilde{f(x)}$ appartient
à $\widetilde k$ puisque ce dernier est algébriquement clos (il hérite cette 
propriété de $k$). Il existe donc $a\in k^\times$ tel que $\abs a=\abs{f(x)}$
et $\widetilde a=\widetilde {f(x)}$. Si l'on pose $g=f-a$ on a alors
$\abs{g(x)}<a$ et donc $\abs g<a$ au voisinage de $x$. 
\end{proof}

\begin{lemm}\label{lemm.degre-positif}
Soit $X$  un espace analytique compact,
soit $T$ un tore 
soit $f\colon X\to T$ un moment; on note $n=\dim(T)$.
Alors $f(X)$ contient $\sq( f_\trop(X)^{(n)})$.
\end{lemm}
\begin{proof}
Soit $C$ un sous-espace compact paralinéaire de dimension~$n$ contenu
dans $f_\trop(X)$; prouvons que $f^{-1}(C_\sq)$ n'est pas vide.
Posons $W=f_\trop^{-1}(C)$;
c'est un domaine 
analytique compact de~$X$ et l'on a $f_\trop(W)=C$.
Il existe donc un point de~$W$ en lequel $f$ est
de dimension tropicale~$n$.
D'après la proposition~\ref{prop.trop.local}, (iii),
$f(W)$ rencontre $S(T)$, donc $C_\sq$ puisque $f_\trop(W)\subset C$.
Ainsi, $f(X)$ rencontre $C_\sq$, ce qu'on voulait démontrer.

Soit $s\in f_\trop(X)^{(n)}$. Soit $C$ un voisinage compact de~$s$
qui est un polyèdre compact de dimension~$n$ contenu dans~$f_\trop(X)$.
D'après ce qui 
précède, $f(X)$ rencontre~$C_\sq$; comme ces voisinages~$C$
forment une base de voisinages de~$s$,
les $C_\sq$ forment une base de voisinages de 
$\sq(s)$, et celui-ci est donc adhérent à~$f(X)$.
Puisque $f(X)$ est compact, $\sq(s)\in f(X)$.
\end{proof}

\begin{prop}\label{prop.projection-tropicale}
Soit $p\colon Y\to X$ un morphisme d'espaces analytiques, 
soit $y\in Y$ et soit $x$ le point 
$p(y)$. 
Soit $g=(g_1,\dots,g_n)$ une famille de fonctions inversibles sur~$Y$,
et soit~$d$ le degré de transcendance sur~$\hrt y$
du corps gradué engendré par la famille $(\tilde g_i(y))$.
Il existe $d$~morphismes de tores $q_1,\dots,q_d\colon \gm^n\to\gm$,
une famille $(V_r)$ de domaines analytiques compacts de~$Y$ contenant~$y$
et dont la réunion est  un voisinage de~$y$,
une famille $(U_r)$ de domaines analytiques compacts de~$X$ 
telle que chaque $U_r$ contienne $p(V_r)$ et, 
pour tout~$r$, une famille $(f_{r,i})$
de fonctions analytiques inversibles sur~$U_r$,
et
des familles $(a_{r,i,j})$ et $(b_{r,i,j})$ de nombres rationnels
telles que 
\[ g_{i,\trop}|_{V_r} = \sum_{j=1}^d a_{r,i,j} (q_j\circ  g)_\trop
+ b_{r,i} (f_{r,i}\circ p)_\trop \]
pour tout couple d'indices $(r,i)$. 

Si le germe~$(X,x)$ est presque bon, on peut supposer que chaque~$U_r$
est un voisinage de~$x$.
\end{prop}

\begin{proof}
La première partie de l'assertion est G-locale sur le germe $(X,x)$; cela permet
de supposer que $X$ est presque bon.

Soit $\mathscr V$ un ouvert 
quasi-compact de $\P_{\hrt y}$ tel que
$\mathscr V\cap \zr{\hrt y}{\widetilde k}=\widetilde{(Y,y)}$. 
Appliquons le théorème~\ref{theo.foncteurs-polyedraux}
à $(\hrt x, \hrt y, \mathscr V, (\widetilde{g_i(y)})_i)$, en reprenant la terminologie et les notations 
introduites en \ref{ss.vocabulaire-foncteur-polyedral}. 
Ce théorème assure l'existence
d'une famille finie $(a_{j\ell m}, \phi_{j\ell m})_{j,\ell,m}$, où les $a_{j\ell m}$ sont des éléments
de $\hrt x^\times$ et où chaque $\phi_{j \ell m}$ est un monôme
$\prod_{i=1}^n x_i^{r_i}$ à exposants rationnels, qui vérifie la propriété suivante, en
désignant pour tout $v\in \P_{\hrt x^\times}$ par $\mathscr D_{j,v}$
le sous-foncteur de $\mathbf G_{v(\hrt x^\times)}^n$
défini par la condition 
$\bigvee_\ell\bigwedge_m v(a_{j\ell m})\phi_{j\ell m}\leq 1$ : \textit{il existe un recuuvremet fini $(\mathscr U_j)$
de $\P_{\hrt x^\times}$ par des parties constructibles tel que pour tout $j$
et tout $v\in \mathscr U_j$ le foncteur polyédral $\mathscr D_{j,v}$ soit de dimension 
$\leq d$ et que pour tout $H\in \mathsf D_{v(\hrt x^\times)}$ et tout $h\in H^n$
les assertions suivantes soient équivalentes : 
\begin{itemize}
\item [(i)] $h\in \mathscr D_{j,v}(H)$ ; 
\item[(ii)] il existe $w\in \mathscr V\cap \zr{\hrt y}{(\hrt x,v)}$ à valeurs  dans un groupe ordonné $H' \in \mathsf D_H$ 
telle que $w(\widetilde{g_i(y)})=h_i$ quel que soit $i$. 
\end{itemize}
}

Pour tout $j$, tout $\ell$ et tout $v$, notons $\mathscr D_{j,\ell,v}$
le sous-foncteur de $\mathbf G_{v(\hrt x^\times)}$, contenu dans
$\mathscr D_{j,v}$, décrit par la condition $\bigwedge_m v(a_{j\ell m})\phi_{j \ell m}\leq 1$. 
Fixons un indice $j$ et une valuation $v\in \mathscr U_j$. Le foncteur polyédral $\mathscr D_{j,v}$ est de dimension 
$\leq d$ et c'est la réunion de ses sous-foncteurs $\mathscr D_{j,\ell,v}$. Chacun d'eux est donc de dimension 
$\leq d$, et est dès lors, au vu de sa forme explicite, contenu dans un translaté d'un sous-foncteur de
$\mathbf G_{v(\hrt x^\times)}$ décrit par une conjonction $\bigwedge_{m\in M_{j\ell,v}}\phi_{j\ell m}\leq 1$, où
$M_{j,\ell,v}$ est un ensemble d'indices de cardinal au moins $n-d$ tel que les $\phi_{j\ell m}$ pour 
$m\in M_{j\ell,v}$ soient $\Q$-multiplicativement indépendants. 

L'ensemble de tous les $M_{j,\ell,v}$ étant fini, il existe $d$ monômes $\psi_1,\ldots, \psi_d$ de la forme $\prod x_i^{e_i}$
avec les $e_i$ dans $\Z$ tels que pour tout $(j,\ell)$ et tout  $v\in \mathscr U_j$ la restriction de
$(\psi_1,\ldots, \psi_d)$ à $\mathscr D_{j\ell, v}$ soit injective. 
Il existe alors pour tout $(j,\ell)$ et tout $i$ entre $1$ et $n$ un élément 
$b_{ij\ell}$ de $\hrt x^\times$ et des nombres rationnels
$r_{ij\ell}, e_{1ij\ell}, \ldots, e_{dij\ell}$ tels que 
la $i$-ème coordonnée $x_i$ 
soit identiquement égale à 
$v(b_{ij\ell})^{r_{ij\ell}}\psi_1^{e_{1ij\ell}}\ldots \psi_d^{e_{dij\ell}}$ 
sur $\mathscr D_{j,\ell,v}$.

Pour tout $i$ entre $1$ et $d$, notons $q_i$ le morphisme de tores de $\gm^n$ dans $\gm$ donné par 
le monôme $\psi_i$, et $\gamma_i$ le moment $q_i\circ (g_1,\ldots, g_n)$. 

Comme $X$ est presque bon il existe un voisinage analytique
$U$ de $x$ dans $X$ et des fonctions analytiques inversibles 
$f_{ij\ell}$ sur $U$ telles que
$b_{ij\ell}=\widetilde{f_{ij\ell}(x)}$ pour tout $(i,j,\ell)$. 

Soit $w\in \widetilde {(Y,y)}$, et soit $H$ un groupe abélien divisible ordonné non trivial
dans lequel $w$ prend ses valeurs. La valuation $w$ appartient à $\mathscr V$, 
et son image $v$ sur $\widetilde{(X,x)}$ appartient à $\mathscr U_j$ pour un certain $j$. 
D'après l'équivalence des 
propriétés (i) et (ii) rappelée plus haut, le $n$-uplet $w(\widetilde{g_i(y)})$ appartient à $\mathscr D_{j,v}(H)$, et
donc à $\mathscr D_{j,v,\ell}(H)$ pour un certain $\ell$. 
Il s'ensuit qu'on a pour tout indice $i$ l'égalité 
\[w(\widetilde {g_i(y)})
=v(\widetilde{f_{ij\ell}(x)}^{r_{ij\ell}})w(\widetilde{\gamma_1(y)})^{e_{1ij\ell}}\ldots w(\widetilde{\gamma_d(y)})^{e_{dij\ell}}.\]

Pour tout $(i,j,\ell)$ notons
$N_{ij\ell}$ le plus petit multiple commun des dénominateurs de
$r_{ij\ell}, e_{1ij\ell}, \ldots, e_{dij\ell}$, et posons
\[g_{ij\ell}=g_i^{N_{ij\ell}}(f_{ij\ell}\circ p)^{-N_{ij\ell}r_{ij\ell}}\gamma_1^{-N_{ij\ell}e_{1ij\ell}}
\ldots \gamma_d^{-N_{ij\ell}e_{dij\ell}}.\]
Par construction, $g_{ij\ell}$ est un moment sur le voisinage $p^{-1}(U)$ de $y$, et si $w$ est une valuation 
de $\widetilde{(Y,y)}$ dont l'image sur $\widetilde{(X,x)}$ appartient à $\mathscr U_j$ 
on a $w(\widetilde{g_{ij\ell}(y)})=1$ ; autrement dit, 
$w\in \zr{\hrt y}{\hrt x}\{\widetilde{g_{ij\ell}(y)}, \widetilde{g_{ij\ell}(y)}^{-1}\}_{1\leq i\leq n}$. 
Posons $\rho_{ij\ell}=\abs{g_{ij\ell}(y)}$
pour tout $i$ compris entre $1$ et $n$. 
Pour tout $(j,\ell)$ notons $V_{j,\ell}$
le domaine analytique de $p^{-1}(U)$ défini comme le lieu de validité des égalités
$\abs{g_{ij\ell}}=\rho_{ij\ell}$ pour $i$ variant entre $1$ et $n$. Le point $y$ appartient
à $V_{j\ell}$ et $\widetilde{(V_{j\ell},y)}$ est égal à 
$\widetilde{(Y,y)}\cap \zr{\hrt y}{\hrt x}\{\widetilde{g_{ij\ell}(y)}, \widetilde{g_{ij\ell}(y)}^{-1}\}_{1\leq i\leq n}$. 
Il s'ensuit que les $\widetilde{(V_{j\ell},y)}$ recouvrent $\widetilde{(Y,y)}$. 
Par conséquent la réunion des $V_{j\ell}$ est un voisinage de $y$ dans $Y$. 
Or on a par construction pour tout $(j,\ell)$ et tout $i$ 
l'égalité
\[\abs {g_i}=\abs{f_{ij\ell}\circ p}^{r_{ij\ell}}
\abs{\gamma_1}^{e_{1ij\ell}}\ldots \abs{\gamma_d}^{e_{dij\ell}}\]
identiquement sur $V_{j\ell}$, ce qui conclut. 
\end{proof}

\section{Compléments sur les tropicalisations}

\begin{prop}\label{prop-trop-pt-abhyankar}
Soit $X$ un bon espace $k$-analytique
purement de dimension~$n$
et 
soit $f\colon X\to \gm^d$ un moment sur $X$.
Il existe un point~$y$ de $X$
tel que $d_k(y)=n$ et $f_\trop(X,y)=f_\trop(X,x)$.
\end{prop}

\begin{proof}
Écrivons $f=(f_1,\ldots,f_d)$, 
et notons $F$ l'extension graduée de $\widetilde k$
dans $\hrt x$ engendrée par les $\widetilde {f_i(x)}$.
On peut démontrer le résultat après avoir remplacé $X$ par n'importe quel
voisinage affinoïde de $x$.
On peut aussi étendre le moment~$f$, 
car
il n'y aura ensuite plus qu'à composer
avec une projection convenable pour obtenir le résultat souhaité pour~$f$. 
Cela permet de supposer qu'il existe un ensemble $I\subset \{1,\ldots, d\}$ tel que
$\widetilde{(X,x)}=\P_{\hrt x/\widetilde k}\{\widetilde{f_i(x)}\}_{i\in I}$. 
Quitte à renuméroter les $f_i$, on peut supposer 
qu'il existe un entier $\delta\in\{0,\dots,n\}$ tel que 
$\widetilde {f_1(x)}, \ldots, \widetilde{f_\delta(x)}$ 
soit une base de transcendance de~$F$ sur~$k$. 
Posons enfin $r_i=\abs{f_i(x)}$ pour tout~$i$. 
Quitte à restreindre~$X$ encore une fois, on peut supposer
que $\abs{f_i}\leq r_i$ sur~$X$ pour tout indice~$i$ appartenant à~$I$. 

Ainsi, la famille~$(f_i)_{i\in I}$ définit un morphisme~$h$ de~$X$ 
vers le polydisque fermé~$D$ en les variables $T_i$ pour $i\in I$ 
et de polyrayon $(r_i)_{i\in I}$. 
Puisque $\abs{f_i(x)}=r_i$ pour tout $i\in I$, 
la réduction $\widetilde {(X,x)}$ est égale à l'image réciproque 
de $\widetilde{(D,h(x))}$ sur $\P_{\hrt x/\widetilde k}$. 
Par conséquent, le morphisme~$h$ est sans bord en~$x$, et est dès lors 
sans bord au voisinage de~$x$.
 
Pour tout $j$ compris entre $\delta+1$ et $d$,
notons $\widetilde P_j$ le polynôme minimal de
$\widetilde {f_j(x)}$ sur $\widetilde k(\widetilde{f_1(x)},\ldots,
\widetilde{f_{j-1}(x)})$;
c'est un polynôme homogène lorsque l'indéterminée~$T$ est affectée 
du degré~$r_j$.  
Choisissons un relevé unitaire~$P_j$ de $\widetilde P_j$ 
dans $\mathscr O_{X,x}$
dont chaque coefficient est une fraction rationnelle en $f_1,\dots,f_{j-1}$. 


La famille $g=(f_1,\ldots, f_\delta)$ peut être vue 
comme un morphisme de~$X$ dans $\A^\delta_k$.
On a $g(x)=\eta_r$;
en vertu du théorème~10.3.7 de \cite{ducros2018},
le morphisme~$g$ est plat en tout point de la fibre $g^{-1}(\eta_r)$,
car $\A^\delta_k$ est réduit.

Comme $\widetilde{P_j} (\widetilde{f_j(x)})=0$ pour tout $j>\delta$,
il existe un voisinage ouvert~$U$ de~$x$ dans~$X$ 
sur lequel le morphisme~$h$ défini ci-dessus est sans bord,
et sur lequel les~$P_j$ sont définis et vérifient les inégalités 
$\abs{P_j(f_j)}<r_j^{\deg(P_j)}$ pour tout~$j>\delta$.
Cet ouvert contient donc un point d'Abhyankar de rang maximal~$y$ 
de la fibre $g^{-1}(x)$. 
Comme $X$ est équidimensionnel, on a $n=\dim_y(Y)$;
par platitude, 
on a $\dim_y (X)=\dim_y (g^{-1}(\eta_r))+\delta$;
par le choix de~$y$, on a  $\dim_y(g^{-1}(\eta_r))= d_{\hr{\eta_r}}(y)$;
on a donc 
\[ n = d_{\hr{\eta_r}}(y)+\delta =d_{\hr{\eta_r}}(y)+d_k(\eta_r) =  d_k(y) .\] 

Démontrons que ce point~$y$ convient.

Soit~$G$ l'extension graduée de~$\widetilde k$ dans $\hrt y$ 
engendrée par les $\widetilde {f_i(y)}$, pour $i\in\{1,\dots,d\}$.
Démontrons par récurrence 
qu'il existe un isomorphisme $F\simeq G$ 
d'extensions graduées de~$\widetilde k$ 
qui envoie $\widetilde {f_i(x)}$
sur $\widetilde {f_i(y)}$ pour tout $i$.  
On commence au cran~$\delta$ : 
comme $g(y)=\eta_r$, on a $\abs{f_i(y)}=r_i$ pour tout $i\in\{1,\dots,\delta\}$,
et les $\widetilde{f_i(y)}$ (pour $1\leq i\leq \delta$) 
sont algébriquement indépendants sur~$\widetilde k$. 
Il existe donc un $\tilde k$-isomorphisme de corps gradués de
$\tilde k(\widetilde{f_1(x)},\dots,\widetilde{f_\delta(x)})$
sur 
$\tilde k(\widetilde{f_1(y)},\dots,\widetilde{f_\delta(y)})$
qui applique 
$\widetilde{f_i(x)} $ sur $\widetilde{f_i(y)}$ pour tout~$i$.
Supposons l'isomorphisme voulu construit jusqu'au cran~$j-1$. 
On a sur tout l'ouvert $U$, 
et en particulier en $y$, l'inégalité $\abs{P_j(f_j)}<r_j^{\deg(P_j)}$,
de sorte que
l'image du polynôme minimal de~$\widetilde{f_j(x)}$
sur 
 $\tilde k(\widetilde{f_1(x)},
\dots, \widetilde{f_{j-1}}(x))$ dans
$\tilde k(\widetilde{f_1(y)},
\dots, \widetilde{f_{j-1}}(y))[T]$
est un polynôme annulateur de $\widetilde{f_j(y)}$, 
qui est nécessairement son polynôme minimal par irréductibilité. 
On en déduit l'isomorphisme cherché au cran~$j$.

En particulier, on a $\abs{f_j(y)}=\abs{f_j(x)}=r_j$ pour tout~$j$,
c'est-à-dire $f_\trop(y)=f_\trop(x)$.

Comme le morphisme~$h$ est sans bord en~$y$, on a 
$\widetilde {(X,y)}=\P_{\hrt y/\widetilde k}\{\widetilde{f_i(y)}\}_{i\in I}$. 
Posons
$\psi=(\widetilde{f_i(y)})_{1\leq i\leq d}$, $\psi'
=(\widetilde{f_i(y)})_{i\in I}$, 
$\phi=(\widetilde{f_i(x)})_{1\leq i\leq d}$
et $\phi'=(\widetilde{f_i(x)})_{i\in I}$. 
On a alors
\begin{align*}
f_\trop(X,y)&=f_\trop(y)\cdot \psi_\trop(\zr{\hrt y}{\widetilde k}
\{\psi'\})(\R_+^\times)\\
&=f_\trop(y)\cdot \psi_\trop(\zr G{\widetilde k}
\{\psi'\})(\R_+^\times)\\
&=f_\trop(x)\cdot \phi_\trop(\zr F{\widetilde k}
\{\phi'\})(\R_+^\times)\\
&=f_\trop(x)\cdot \phi_\trop(\zr {\hrt x}{\widetilde k}
\{\phi'\})(\R_+^\times)\\
&=f_\trop(X,x),
\end{align*}
où la première et la dernière égalités proviennent
de la remarque~\ref{rem-description-trop-germ}, où la seconde 
et la quatrième égalités proviennent de la surjectivité 
des applications $\zr{\hrt y}{\widetilde k}
\to \zr G{\widetilde k}$
et $\zr{\hrt }{\widetilde k}
\to \zr F{\widetilde k}$
respectivement, 
et où la troisième égalité repose sur l'isomorphisme
construit plus haut entre $F$ et $G$ et le fait que
$f_\trop(y)=f_\trop(x)$. 
\end{proof}

\subsection{}
Soit $X$ un espace $k$-analytique, soit $f=(f_1,\ldots, f_d)$ un moment sur $X$ et soit $x$ un point intérieur de $X$. 
Soit~$F$ l'extension graduée de $\widetilde k$
dans $\hrt x$ engendrée par les $\widetilde {f_i(x)}$. 
Soit $\delta$ son degré de transcendance (gradué) sur~$\widetilde k$.
D'après la remarque~\ref{rem-description-trop-germ},
l'image par~$f_\trop$ du germe $(X,x)$ est un germe de polytope
purement de dimension~$\delta$ qui
ne dépend que de la classe d'isomorphie de~$(F, (\widetilde{f_i(x)}))$.

 \begin{prop}
 Soit $X$ un espace $k$-analytique sans bord, soit $f$ un moment sur $X$ et soit $x$ un point de $X$. Soit~$\delta$ la dimension tropicale de $f$ en $x$. 
Il existe un point~$y$ de $X$ et un morphisme~$g$ de $(X,y)$ 
vers un germe $(Y,z)$ tels que : 
 $f_\trop(g^{-1}(z),y)=f_\trop(X,x)$ et $\dim_y g^{-1}(z)=\delta$. 
  \end{prop}
  
  \begin{proof}
  Choisissons un point d'Abhyankar de rang maximal $y$ de $X$
 tel que $f_\trop(X,y)=f_\trop(X,x)$
(proposition~\ref{prop-trop-pt-abhyankar}). 
 Écrivons $f=(f_1,\ldots, f_d)$, et notons $F$ le sous-corps gradué de
 $\hrt y$ engendré par $\widetilde k$ et les $\widetilde{f_i(y)}$. 
 Quitte à renuméroter les $f_i$, on peut supposer que $(\widetilde{f_1(y)}, 
 \ldots,\widetilde {f_\delta(y)})$ est une base de transcendance graduée
de~$F$ sur~$\widetilde k$. 
 Soit~$n$ la dimension de~$X$ en~$y$. Choisissons une famille $(g_{\delta+1},\ldots, g_n)$ de
 fonctions analytiques définies au voisinage de~$y$ et telles que
 $(\widetilde{f_1(y)}, 
 \ldots,\widetilde {f_\delta(y)},
 \widetilde{g_{\delta+1}(y)},\ldots, \widetilde{g_n(y)}$ soit une base de transcendance de~$\hrt y$ sur~$\widetilde k$. 
 
Posons $r=(\abs{g_\delta+1}(y),\ldots,\abs{g_n(y)})$ et
soit $g$ le morphisme $(g_{\delta+1},\ldots, g_n)$ de $(X,y)$ vers $(\A^{n-\delta}, \eta_r)$.
Le théorème~10.3.7 de \cite{ducros2018}
assure que $g$ est plat en tout point de $g^{-1}(\eta_r)$ ; il s'ensuit
que $\dim_y g^{-1}(\eta_r)=n-(n-\delta)=\delta$. 

Le but est maintenant de montrer que $g$ convient; c'est-à-dire que
$f_\trop(g^{-1}(\eta_r),y))=f_\trop(X,y)$.
Pour cela, posons $t_i=\widetilde{g_i(y)}$ pour $i$ entre 
$\delta+1$ et $n$. Les $t_i$ forment une famille algébriquement indépendants sur $\widetilde k$, et le sous-corps
gradué $\hrt z$ de $\hrt y$ s'identifie à $\widetilde k(t_i)_{\delta+1\leq
i\leq n}$ ; quant au sous-corps de $\hrt y$ engendré par les
$\widetilde{f_1(y)}, \ldots, \widetilde {f_d(y)}$ et $\hrt z$, il s'identifie à $F(t_i)_{\delta+1\leq
i\leq n}$, les $t_i$ restant algébriquement indépendants sur $F$. 

La remarque~\ref{rem-description-trop-germ} décrit explicitement
les tropicalisations $f_\trop (X,y)$ et 
$f_\trop(g^{-1}(\eta_r),y)$ en termes de cônes.
Il suffit ainsi de démontrer que pour tout groupe abélien ordonné divisible non trivial~$H$ et tout
$d$-uplet
$(h_1,\ldots, h_d)$ d'éléments de~$H$,
les assertions suivantes sont équivalentes : 
\begin{itemize}
\item[(i)] il existe une valuation graduée $\abs \cdot $ sur $F$, triviale sur $\widetilde k$, telle que
$\abs{\widetilde{f_i(y)}}=h_i$ pour tout $i$ entre 
$1$ et $d$ ; 
\item[(ii)]  il existe une valuation graduée $\abs \cdot$ sur $F(t_i)_{\delta+1\leq
i\leq n}$, triviale sur $\widetilde k(t_i)_{\delta+1\leq
i\leq n}$, telle que $\abs{\widetilde{f_i(y)}}=h_i$ pour tout $i$ entre 
$1$ et $d$.
\end{itemize}
Il est clair que (ii)$\Rightarrow$(i). Et si (i) est vraie, on montre (ii) en munissant $F(t_i)_{\delta+1\leq
i\leq n}$ de la valuation de Gau\ss\ prolongeant $\abs \cdot $ et valant~$1$ sur chacun des~$t_i$. 
\end{proof}

\begin{prop}
Soit $X$ un espace $k$-analytique purement de dimension~$n$
et soit $f\colon X \to T$ un moment.
Supposons qu'il existe un entier $r\geq 0$ tel que le 
lieu des points de $X$ en lesquels $f$ est de dimension 
relative $r$ soit dense (c'est automatique si $X$ est irréductible).
\begin{enumerate}
\item
Pour tout $x\in X$, on a $\dim_{\trop,x}(f)\leq n-\dim(f^{-1}(f(x)))$.
En particulier, $\dim(f_\trop(X))\leq n-r$.
\item 
Pour tout point $\xi\in f_\trop(X)$ tel que $\xi\not\in f_\trop(\partial(X))$,
on a $\dim_\xi(f_\trop(X))=n-r$. 
\item 
Si $X$ est affinoïde et $r>0$, alors $f_\trop(X)=f_\trop(\partial(X))$.
\end{enumerate}
\end{prop}

Les points~\emph a) et~\emph b) de cet énoncé 
sont une généralisation du théorème~1.1 de~\cite{gubler2007a}
qui traite le cas où $X$ est un sous-espace analytique  fermé
d'un domaine analytique de~$\gm^n$ de la forme~$\trop^{-1}(\Delta)$,
où $\Delta$ est un polyèdre compact de~$\R^n$;
la preuve est similaire.
Compte tenu des deux premiers points, l'assertion~\emph c) est 
le théorème 6.6 de~\cite{GublerRabinoff-2025a}, 
mais notre approche en redonne une preuve.

\begin{proof}
\begin{enumerate}
\item
Soit $x\in X$ et soit $t$ son image sur $T$. 
Soit $s$ la dimension de la fibre $(f^{-1}(t))$;
on a $r\leq s$ par hypothèse.

Soit $z\in f^{-1}(t)$ tel que $d_{\hr t}(z)=s$;
on a $d_k(t)+s=d_k(z)\leq n$, si bien que $d_k(t)\leq n-s$.
Le sous-corps gradué de~$\hrt x$ engendré par les~$\widetilde{f_i(x)}$
est contenu dans~$\hrt t$, donc est de degré de transcendance gradué
$\leq n-s$. Par suite, la dimension tropicale de~$f$
en~$x$ est majorée par $n-s$.

\item
Nous allons démontrer l'énoncé suivant :
pour tout $\xi\in f_\trop(X)\setminus f_\trop(\partial(X))$,
on a $\dim_\xi(f_\trop(X))=n-r$ et $\dim_\xi(f_\trop(X))=n$ si $X$ est affinoïde.
Cela entraîne l'assertion~\emph b) du théorème;
quand $X$ est affinoïde, cela entraîne également
que $n-r=n$ si $f_\trop(X)\neq f_\trop(\partial(X))$, d'où~\emph c).

Si $K$ est une extension complète de $k$ alors
$\partial (X_K)$ est contenu dans l'image réciproque de $\partial(X)$, 
si bien qu'un 
point~$\xi$ de $f_\trop(X)$ qui n'appartient pas à $f_\trop(\partial X)$ 
n'appartient pas non plus à $\trop(X_K)$. 
On peut donc démontrer l'assertion après extension des scalaires à~$K$ ; 
on peut dès lors supposer que $k$ est algébriquement clos,
non trivialement valué,
et que $X$ et $f_\trop^{-1}(\xi)$ sont stricts.

Soit $d=\dim_\xi(f_\trop(X))$. On raisonne par récurrence sur~$d$. 
On suppose donc l'énoncé établi en dimension~$<d$. 
La dimension $\dim_\xi (f_\trop(X))$ est le maximum des $\dim_\xi f_\trop(Y)$,
où $Y$ parcourt l'ensemble~$E$ des composantes irréductibles réduites de~$X$ 
dont l'image par~$f_\trop$ contient~$\xi$. 
De plus, $\partial (Y)=\partial(X)\cap Y$ pour toute composante $Y\in E$,  
il suffit de démontrer l'énoncé pour toute composante~$Y\in E$.
Si $\dim_\xi(f_\trop(Y))<d$, l'assertion est vraie par l'hypothèse de
récurrence.
On s'est ainsi ramené au cas où $X$ est irréductible et réduit. 

Supposons $d=0$. 
Dans ce cas, $\xi$ est isolé dans~$f_\trop(X)$, si bien 
que $f_\trop(X)=\{\xi\}$ puisque, $X$ étant irréductible, et donc connexe,
sa tropicalisation~$f_\trop(X)$ est connexe.
Comme on suppose que $\xi\notin f_\trop(\partial(X))$, 
l'espace~$X$ est sans bord ; 
puisqu'il est compact, il est propre sur~$k$. 
La $k$-algèbre $\mathscr O_X(X)$ est alors de dimension finie (par propreté)
et est intègre car $X$ est irréductible et réduit ; 
comme $k$ est algébriquement clos, il vient $\mathscr O_X(X)=k$.
Ainsi,  le moment~$f$ est constant, donc  $r=n$, et l'on a bien $d=0=n-r$. 
Supposons de plus que $X$ est affinoïde. 
Alors $X$ est à la fois affinoïde et propre, donc fini, si bien que $n=0$ 
et donc que $d=n$. 

Supposons maintenant $d>0$.
L'ensemble $f_\trop(X)$ est un sous-espace paralinéaire compact
de~$T_\trop$; choisissons-en une décomposition cellulaire (finie)~$\mathscr C$.
Comme $d>0$, le point~$\xi$ n'est pas isolé dans~$f_\trop(X)$
et toute cellule maximale de~$\mathscr C$ contenant~$\xi$
est de dimension~$>0$.
Il existe alors un hyperplan ($\Q$-rationnel)~$H$ de~$T_\trop$ contenant~$\xi$
et tel que $\dim(H\cap C)=\dim(C)-1$ 
pour toute cellule maximale $C\in\mathscr C$ qui contiennent~$\xi$.
Écrivons une équation affine de~$H$ de la forme $\sum a_i t_i=b$,
où les~$a_i$ appartiennent à~$\Z$
et posons $h=\prod f_i^{a_i}$.

L'hypothèse que les cellules maximales passant par~$\xi$
ne sont pas contenues dans~$H$ assure que $h_\trop$ n'est pas constante sur~$X$.
Pour tout $\lambda \in k$, le fermé de Zariski $V(h-\lambda)$ 
de l'espace irréductible~$X$
est donc d'intérieur vide;
d'après le Hauptidealsatz, 
il est purement de dimension~$n-1$.

Soit $U_f$ l'ouvert de Zariski de~$X$ tel que $\dim_x(f^{-1}(f(x)))=r$
pour tout $x\in U_f$; par hypothèse, il est dense. 
Soit $Z_f$ son complémentaire.
Soit $\lambda\in k$; les composantes irréductibles de~$V(h-\lambda)$
sont de dimension~$n-1$, de sorte que si l'une d'elles est contenue
dans le fermé d'intérieur vide~$Z_f$, c'en est une composante irréductible.
D'autre part, les $V(h-\lambda)$ sont deux à deux disjoints.
Puisque $Z_f$ n'a qu'un nombre fini de composantes
irréductibles, l'ensemble~$\Lambda$ des $\lambda\in k$
tels qu'une composante irréductible de $V(h-\lambda)$ soit contenue dans~$Z_f$
est fini.

Le domaine analytique $f_\trop^{-1}(\xi)$ de~$X$ n'est pas vide,
il rencontre donc~$U_f$. 
La la réunion, pour $\lambda\in\Lambda$, 
des $h^{-1}(\lambda)\cap U_f\cap f_\trop^{-1}(\xi)$,
est purement de dimension~$n-1$. 

Son complémentaire dans $U_f\cap f_\trop^{-1}(\xi)$ est donc dense ; 
comme $k$ est algébriquement clos, non trivialement valué,
et comme~$f_\trop^{-1}(\xi)$ est strictement analytique, 
il existe un $k$-point~$x\in U_f\cap f_\trop^{-1}(\xi)$
qui n'appartient à aucune hypersurface $h^{-1}(\lambda)$,
pour $\lambda\in \Lambda$. 
Posons $\lambda=h(x)$.

Soit $Y$ une composante irréductible de~$V(h-\lambda)$
et soit $Y'$ le complémentaire dans~$Y$ de la réunion
des autres composantes.
Par le choix de~$\lambda$, on a $Y\cap U_f\neq\emptyset$.
Pour tout $y\in Y'\cap U_f$, 
on a $ \dim_y(f|_Y^{-1}(f(y))) = \dim_y(f|_{V(h-\lambda)}^{-1}(f(y)))$,
car $y$ n'appartient qu'à une seule composante de~$V(h-\lambda)$.
D'autre part, $h$ est un monôme en les~$f_i$, si bien que
$f|_{V(h-\lambda)}^{-1}(f(y)) = 
f^{-1}(f(y))$.
Cela démontre que la dimension relative générique de $f|_Y$ est égale à~$r$.

Soit $Y$ une composante irréductible de~$V(h-\lambda)$
telle que $\xi\in f_\trop(Y)$; c'est par exemple le cas
des composantes irréductibles de~$V(h-\lambda)$ qui contiennent~$x$.
Comme $Y$  est un fermé de Zariski de~$X$, 
on a $\partial(Y)\subset\partial(X)$,
donc $\xi\not\in f_\trop(\partial(Y))$.
Puisque $f_\trop(Y)\subset H$, 
on a on a $\dim_\xi(f_\trop(Y))\leq d-1$.
On peut alors appliquer l'hypothèse de récurrence à~$f|_Y$;
on en déduit que
que $\dim_\xi(f_\trop(Y))=\dim(Y)-r=n-1-r$;
par suite, $n-1-r\leq d-1$, d'où $n-r\leq d$. D'après le~\emph a),
il vient $d=n-r$.

Si $X$ est affinoïde, $Y$ l'est aussi, 
et l'hypothèse de récurrence assure
que $\dim_\xi(f_\trop(Y))=\dim(Y)$, si bien que $d=n$.
\qedhere
\end{enumerate}
\end{proof}

\begin{prop}\label{prop.dim-trop-d}
On suppose que la valuation de~$k$ n'est pas triviale.
Soit $X$ un espace strictement $k$-analytique compact
et soit $f\colon X\to T$ un moment.
Si $f_\trop(X)$ est de dimension~$d$, il existe un point~$x\in \Int(X)$
en lequel $f$ est de dimension tropicale~$d$.
\end{prop}
\begin{proof}
Quitte à remplacer~$f$ par~$p\circ f$, où $p\colon T\to\gm^d$ 
est un morphisme de tores convenable, on se ramène au cas où $T=\gm^d$.
Soit~$x$ un point de~$X$ en lequel $f$ est de dimension tropicale~$d$;
il en existe car $X$ est compact,
d'après la relation~\eqref{eq.dimtrop-sup}.
D'après la proposition~\ref{prop.trop.local},
$f(x)$ appartient au squelette de~$\gm^d$, c'est-à-dire $x\in\Sigma_f$.

Démontrons qu'il existe un domaine strictement analytique 
compact~$V$ de~$X$, contenant~$x$,
et un domaine analytique compact~$W$ 
de~$\gm^d$, tel que le morphisme~$f$
induise un morphisme plat de~$V$ sur~$W$
(au sens de~\cite{ducros2018}, definition~4.1.8;
comme les espaces sont stricts, cela équivaut en fait à
la notion de morphisme rig-plat en géométrie rigide, 
\loccit, corollary~{10.3.3}).
Le lieu des points de~$X$ en lesquels $f$ est plat est ouvert
(de Zariski), d'après \loccit, théorème~{10.3.2}, et contient le
point~$x$ (\loccit, théorème~10.3.7). Il contient
donc un voisinage strictement analytique compact~$V$ de~$x$.
L'image~$W$ de~$V$ est un domaine strictement analytique compact de~$\gm^d$
(\loccit,  théorème~9.2.1; c'est dans ce cas un théorème de Raynaud,
\cite{bosch-lutkebohmert1993b}, théorème~5.11).

Il existe alors un espace strictement analytique compact~$U$
et un morphisme $j\colon U\to V$ tel que la composée $g=f\circ j\colon U\to W$
soit surjective et de dimension relative nulle
(\cite{ducros2018}, théorème~9.1.3).
Puisque $f$ est de dimension tropicale~$d$ en~$x$ et que $V$
est un voisinage de~$x$, on a $\dim(f_\trop(V))=d$.
Or $f_\trop(V)=\trop(W)=g_\trop(U)$, si bien que $\dim(g_\trop(U))=d$.
Par ailleurs, $\dim(U)=d$, puisque $g$ est de dimension relative nulle.
En particulier, $\Sigma_g^{(d)}$ n'est pas vide.
D'après le théorème~\ref{prop.trop.global},
$g_\trop(\partial(U))$ est contenu dans un polyèdre
de dimension~$\leq (d-1)$, si bien que $\Sigma_g^{(d)}$ rencontre~$\Int(U)$.
Soit $u\in \Int(U)\cap\Sigma_g^{(d)}$ et soit $x=j(u)$; on a $x\in\Int(V)$,
donc $x\in\Int(X)$.
De plus, on a $\dim_{\trop,u}(g)=d$, donc $\dim_{\trop,x}(f)= d$.
\end{proof}

\begin{lemm}\label{lemm.moment-bord}
Supposons que la valeur absolue du corps~$k$ n'est pas triviale.
Soit~$X$ un espace analytique localement holomorphiquement séparé.
Pour tout point $x\in X$, il existe un voisinage analytique compact~$V$
de~$x$ et un moment $f\colon V\to T$ tel que $f_\trop(x)\not\in f_\trop(\partial(V/X))$.
\end{lemm}
\begin{proof}
Soit $W$ un voisinage analytique compact de~$x$ qui est holomorphiquement
séparé. 

Soit $y\in W$; démontrons qu'il existe une fonction~$f_y\in\mathscr O(W)$,
inversible en~$x$, telle que $\abs{f_y(y)}\neq\abs{f_y(x)}$.
Par définition d'un espace holomorphiquement séparé, il existe
$f\in\mathscr O(W)$ tel que  $\abs{f(y)}\neq\abs{f(x)}$.
Si $f(x)\neq 0$, la fonction~$f$ convient. Supposons $f(x)=0$;
alors $f(y)\neq 0$ et,
comme la valeur absolue de~$k$ n'est pas triviale, il existe
$\lambda\in k^\times$ tel que $\abs\lambda <\abs{f(y)}$;
la fonction $f+\lambda$ convient.

Notons~$W_y$ l'ensemble défini par l'inégalité
$\abs{f_y}\neq \abs{f_y(x)}$; c'est un voisinage de~$y$ dans~$W$.
Comme $\partial(W/X)$ est compact, il existe une famille
finie $(y_1,\dots,y_n)$ dans~$\partial(W/X)$ telle
que les $W_{y_i}$ recouvrent~$\partial(W/X)$.
Soit $r=\frac12 \inf \abs{f_{y_i}(x)}$ et soit $V$ le domaine
analytique compact de~$W$ défini par les conditions
$\abs{f_{y_i}}\geq r$. On a $\abs{f_{y_i}(x)}>r$ pour tout~$i$,
de sorte que $V$ est un voisinage de~$x$ dans~$W$,
et a fortiori dans~$X$. 
Pour tout~$i$, notons $f_i$ la fonction~$f_{y_i}|_V$; elle est inversible.

Soit $f\colon V\to \gm^n$ le moment $(f_1,\dots,f_n)$.
Démontrons qu'il vérifie la condition requise par le lemme.
Soit $y\in\partial(V/X)$. Il y a deux possibilités.
Ou bien il existe $i\in\{1,\dots,n\}$ tel que $\abs{f_i(y)}=r$;
dans ce cas, $f_\trop(y)\neq f_\trop(x)$.
Ou bien, $y\in\partial(W/X)$; dans ce cas, il existe~$i\in\{1,\dots,n\}$
tel que $y\in W_{y_i}$ et $\abs{f_i(y)}\neq \abs{f_i(x)}$, par construction de~$W_{y_i}$, donc $f_\trop(y)\neq f_\trop(x)$.
\end{proof}

\begin{rema}
Le lemme est faux lorsque la valeur absolue du corps~$k$ est triviale.
Prenons pour espace~$X$ le disque unité ouvert et point~$x$ l'origine;
l'application $r\mapsto \eta_r$ induit un homéomorphisme
de $[0;1\mathopen[$ sur~$X$, qui applique~$0$ sur le point~$x$.
Soit $V$ un voisinage analytique compact de~$x$ dans~$X$.
La composante connexe de~$x$ dans~$V$ est l'image d'un intervalle 
de la forme $[0;r]$, pour  $r\in\mathopen]0;1\mathclose[$;
c'est donc le disque fermé de rayon~$r$ et de centre~$x$.
Sur cet espace, toute fonction inversible est de norme constante,
donc ne sépare pas~$x$ du bord~$\eta_r$ de cette composante connexe.
\end{rema}

\section{Structure G-tropicale affinée sur un espace analytique}

\label{s.G-trop-aff}
Soit $X$ un espace $k$-analytique topologiquement séparé.\footnote{%
Par définition, un espace G-tropical est topologiquement séparé ;
dans la suite de ce texte, nous avons pris garde à indiquer cette
hypothèse lorsqu'elle était nécessaire, tout éventuel oubli ne serait
qu'inadvertance.}
Dans ce paragraphe, nous munissons l'espace~$X$
d'une structure d'espace G-tropical (définition~\ref{defi.G-tropical})
affiné (\S\ref{sec.G-trop-aff}); cela sera toujours cette structure
que nous considérerons.
Les structures paralinéaires considérées seront relatives
au couple $(\Q,\R)$. En particulier, l'espace~$\R^n$
est muni de sa $\Q$-structure standard et les polyèdres de~$\R^n$
sont définis  par des inégalités dont la partie linéaire est $\Q$-rationnelle.

\begin{prop}
Soit $X$ un espace $k$-analytique topologiquement séparé.
Lorsqu'on prend pour domaines de~$X$ ses domaines analytiques
et pour G-faisceau des fonctions paralinéaires
le G-faisceau en $\Q$-espaces vectoriels
engendré par les fonctions constantes et celles de la forme $f_\trop$,
on munit~$X$ d'une structure d'espace G-tropical.
\end{prop}
Nous munirons implicitement tout espace $k$-analytique séparé
de la structure d'espace G-tropical décrite dans cette proposition.
\begin{proof}
L'espace~$X$ a été supposé topologiquement séparé;
il est localement compact et possède une base d'ouverts
dénombrables à l'infini (lemme~\ref{lemm.vois-paracompact}).

Les fonctions paralinéaires sur~$X$ sont des fonctions continues.

Vérifions les autres axiomes des espaces G-tropicaux.
Soit $W$ un domaine analytique de~$X$
et soit $f\colon W \to\R^m$ une application paralinéaire.
\begin{itemize}
\item Si $W$ est compact, alors $f(W)$ est un polyèdre compact
de~$\R^m$. On peut en effet raisonner G-localement sur~$W$,
ce qui ramène au cas où $f$ est de la forme $q\circ h_\trop$,
où $h\colon W\to\gm^n$ est un moment sur~$W$ et $q\colon \R^n\to\R^m$ 
est une fonction affine à partie
linéaire rationnelle. 
L'assertion se déduit du théorème~\ref{prop.trop.global}.
\item Si $P$ est un sous-espace paralinéaire de~$\R^m$ contenant~$f(W)$
et si $\phi\colon P\to\R$ est paralinéaire, alors $\phi\circ f$
est paralinéaire. On peut encore raisonner localement,
et est évidente lrosque $f$ est de la forme $q\circ h_\trop$ comme ci-dessus.
\item Si $P$ est un sous-espace paralinéaire de~$\R^m$,
alors $f^{-1}(P)$ est un domaine de~$W$. On peut encore raisonner G-localement,
donc supposer que $f$ est de la forme $q\circ h_\trop$ comme ci-dessus;
ainsi $f^{-1}(P) = h_\trop^{-1}(q^{-1}(P))$. 
Comme $q$ est à partie linéaire $\Q$-rationnelle, $q^{-1}(P)$ 
est un polyèdre de~$\R^n$.  Ainsi, $f^{-1}(P)$ est défini
par des inégalités de la forme $\abs{u}\leq \lambda$,
où $u$ est une fonction inversible sur~$W$ et $\lambda$ un nombre réel strictement positif, donc est un domaine analytique.
\qedhere
\end{itemize}
\end{proof}

\subsection{}
Soit $X$ un espace $k$-analytique topologiquement séparé,
considéré comme espace G-tropical,
soit $f\colon X\to T$ un moment et soit $x$ un point de~$X$.
La dimension tropicale  de~$f$ en~$x$ coïncide avec la dimension
tropicale de la carte G-tropicale $f_\trop$.

Il en résulte que pour toute carte G-tropicale~$f$ sur~$X$,
on a l'inégalité 
\begin{equation}
\dim_{\trop,x}(f) \leq d_k(x) \leq \dim_x(X). 
\end{equation}

\begin{lemm}
Soit $X$ un espace analytique topologiquement séparé 
et soit~$P$ une partie fidèlement paralinéaire de~$X$.
Il existe un G-recouvrement $(P_i)$ de~$P$ par des parties paralinéaires et,
pour tout~$i$, 
un moment $f_i$ sur un domaine~$V_i$ de~$X$ contenant~$P_i$, tels que 
$f_{i,\trop}$ soit une carte paralinéaire sur~$P_i$.
\end{lemm}
\begin{proof}
Soit $f\in\Lambda(P)$ une carte paralinéaire sur~$P$.
On peut raisonner G-localement, et donc supposer
que $f$ est de la forme $q\circ h_\trop$, où $h$ est un moment sur~$X$
et $q$ une application affine. Alors, $h_\trop\in\Lambda(P)$.
\end{proof}
\begin{lemm}\label{lemm.bord-squelette-negligeable}
Soit $X$ un espace $k$-analytique topologiquement séparé,
de dimension~$\leq n$.
Pour toute partie paralinéaire~$P$ de~$X$, 
on a $\dim(P)\leq n$ et l'ensemble
$\partial(X)\cap P$ est $n$-négligeable
au sens de la définition~\ref{PL-negligeable}.
\end{lemm}
\begin{proof}
L'assertion est G-locale. Cela permet de supposer que $X$
est compact et qu'il existe
un moment~$f$ sur~$X$ tel que $f_\trop\in\Lambda(P)$.
D'après le corollaire~\ref{coro.dim-ftrop},
$f_\trop(X)$ est une partie paralinéaire de dimension~$\leq n$,
qui contient $f_\trop(P)$; par suite, $P$ est de dimension~$\leq n$.
D'après le corollaire~\ref{coro.dim-ftrop-bord},
$f_\trop(\partial(X))$ est  contenu dans une partie
paralinéaire de dimension~$\leq (n-1)$;
a fortiori, $\partial(X)\cap P$ est $n$-négligeable.
\end{proof}
\subsection{}
Un morphisme $p\colon Y\to X$ d'espaces analytiques topologiquement
séparés, non nécessairement
définis sur le même corps, induit un morphisme d'espaces G-tropicaux.

\subsection{}\label{ss.an-esp-affine}
Soit $X$ un espace $k$-analytique topologiquement séparé;
définissons maintenant une structure d'espace G-tropical affiné sur~$X$.

Soit $u$ une fonction numérique définie
sur un domaine~$V$ de~$X$.
On dit que $u$ est \emph{affine}\index{fonction affine sur un espace analytique}
au voisinage d'un point~$x$ de~$V$
s'il existe un voisinage~$U$ de~$x$ dans~$V$,
un moment~$f$ sur~$U$, un nombre rationnel~$a$ et un nombre réel~$c$
tels que $u=a f_\trop+c$ sur~$U$.
On dit que $u$ est affine si elle l'est au voisinage de tout point de~$V$.

Les fonctions affines sur un domaine de~$X$
forment un sous-préfaisceau du G-faisceau
des fonctions paralinéaires.
Par définition d'une fonction affine, 
les fonctions affines définissent un  faisceau
pour la topologie usuelle;
par définition d'une fonction paralinéaire, 
toute fonction paralinéaire est G-localement affine.

On a ainsi défini une structure d'espace G-tropical affiné sur~$X$
au sens du~\S\ref{sec.G-trop-aff}.

\begin{prop}\label{prop.presque-bon-decent}
Soit $X$ un espace analytique topologiquement séparé
et soit $x$ un point de~$X$.
Si $X$ est presque bon en~$x$, alors l'espace G-tropical affiné associé à~$X$
est décent en~$x$ (définition~\ref{defi.espace-decent}).
\end{prop}
\begin{proof}
Soit $V$ un domaine analytique de $X$ contenant $x$. 

Montrons tout d'abord que si $f$ est une fonction affine sur $V$, il existe un voisinage ouvert 
$U$ de $x$ dans $X$ et un prolongement affine de $f|_{V\cap U}$ à $U$. 
Pour ce faire, on peut remplacer $V$ par n'importe quel voisinage domanial de $x$
dans $V$, et donc supposer que $f=ag_\trop+b$ où $a\in \Q$, où $g$ est un moment sur $V$ et où $b\in \R$. 
Puisque $X$ est presque bon en $x$, il existe $h\in \mathscr O_{X,x}^\times$ telle que $\abs{h(x)-g(x)}<\abs{g(x)}$. 
Choisissons un voisinage ouvert $U$ de $x$ sur lequel $h$ est définie et tel que
$\abs{h-g}<\abs g$ sur $U\cap V$. 
On a alors $h_\trop=g_\trop$ sur $U\cap V$, et $ah_\trop+b$ est un prolongement de $f|_{U\cap V}$ à $U$ tout entier. 

Comme $X$ et presque bon en $x$, il est séparé en $x$, et $\widetilde{(V,x)}$ est alors un ouvert quasi-compact
et non vide de $\widetilde{(X,x)}$. On peut donc décrire 
\[\widetilde{(V,x)}=\bigcup_i \zr{\hrt x}{\widetilde k}\{f_{ij}\}_j\] où les $f_{ij}$ sont des éléments de $\widetilde k^\times$. 
Par presque bonté de $X$ en $x$, il existe une famille $(g_{ij})$ d'éléments de $\mathscr O_{X,x}^\times$ telle que
$\widetilde{g_{ij}(x)}=f_{ij}$ pour tout $(i,j)$ ; posons $r_{ij}=\abs{g_{ij}(x)}$. 
Choisissons un voisinage ouvert $U$ de $x$ sur lequel les $g_{ij}$ sont définis. 
La description de $\widetilde{(V,x)}$ assure alors que celui-ci coïncide avec le domaine de 
$(U,x)$ défini par la condition $\bigvee_i\bigwedge_j \abs{g_{ij}}\leq r_{ij}$ ; autrement dit, si l'on restreint suffisamment 
$U$ alors $V\cap U$ est le domaine de $U$ défini par la condition $\bigvee_i\bigwedge_j \abs{g_{ij}}\leq r_{ij}$. Il s'ensuit
que $X$ est décent en~$x$. 
\end{proof}

\section{Quelques classes de fonctions paralinéaires}

\subsection{}
Soit $X$ un espace $k$-analytique topologiquement séparé. 
Un \Gstr -recouvrement de~$X$ est un G-recouvrement~$(X_i)$ de~$X$ 
par des domaines analytiques stricts
strict 
(au sens de ce texte, qui diffère du sens classique si $k$ est trivialement valué, 
\cf \ref{terminologie-stricte}) ; 
un tel recouvrement existe si et seulement si $X$ lui-même est strict. Si c'est le cas, on note
$X_{\groth^{\mathrm{str}}}$ le site dont la
la catégorie sous-jacente est celle des domaines $k$-analytiques stricts de $X$
(les flèches étant les inclusions) et dont les familles couvrantes sont les \Gstr -recouvrements ; 
sa topologie est appelée la \Gstr -topologie. 

\subsection{}\label{ss-classes-pl}
Notons $\PL_X^\Z$ le sous-faisceau de $\PL_X$ sur $X_\groth$ qui associe à un domaine
analytique $V$ de $X$ l'ensemble des fonctions paralinéaires sur $V$ qui sont G-localement 
de la forme 
$\log \abs g+c$ où $g$ est une fonction analytique inversible et $c$ une constante réelle. 
Ses sections sont appelées fonctions paralinéaires \textit{entières}.

Supposons de plus que $X$ est strict. 
On note $\PL_X^{\mathrm{str}}$
le sous-faisceau de $\PL_X|_{X_{\groth^{\mathrm{str}}}}$
sur $X_{\groth^{\mathrm{str}}}$
qui associe à un domaine
strictement analytique $V$ de $X$ l'ensemble des fonctions paralinéaires sur $V$ qui sont \Gstr -localement 
de la forme 
$\lambda \log \abs g$ où $g$ est une fonction analytique inversible et $\lambda$ un nombre rationnel ; 
ses sections sont appelées fonctions paralinéaires \textit{strictes}. Et l'on note 
$\PL_X^{\mathrm{str},\Z}$ le sous-faisceau de $\PL_X^{\mathrm{str}}$ sur $X_{\groth^{\mathrm{str}}}$
qui associe à un domaine
strictement analytique $V$ de $X$ l'ensemble des fonctions paralinéaires sur $V$ qui sont \Gstr -localement 
de la forme 
$\log \abs g$ où $g$ est une fonction analytique inversible.  Ses sections sont appelées fonctions paralinéaires \textit{strictement entières}.

\begin{lemm}\label{lemm.pl-inf-sup}
Le sous-faisceau $\PL_X^\Z$ de $\PL_X$ en est un sous-faisceau en groupes abéliens, stable 
sous $\min$ et $\max$. Si  $X$ est strict, les sous-faisceaux 
$\PL_X^{\mathrm{str}}$ et $\PL_X^{\mathrm{str},\Z}$ de $\PL_X|_{X_{\groth^{\mathrm{str}}}}$
en sont des sous-faisceaux en groupes abéliens, stables
sous $\min$ et $\max$. 

\end{lemm}
\begin{proof}
Nous allons traiter le cas de $\PL_X^{\mathrm{str},\Z}$  ; les autres sont établis de manière analogue. 
On suppose donc que $X$ est strict, et l'on se donne un domaine analytique strict $V$ de $X$. 
Il est clair que $0\in \PL_X^{\mathrm{str},\Z}(V)$. Soit $f$ et $g$ deux fonctions appartenant à 
$\PL_X^{\mathrm{str},\Z}(V)$.
Soit $(V_i)$ et $(W_j)$ deux \Gstr -recouvrements de $V$ tels $f|_{V_i}$ soit pour tout $i$ égale
$\log \abs{f_i}$ pour une certaine fonction analytique inversible $f_i$ sur $X_i$, et
que $g|_{W_ij}$ soit pour tout $j$ égale à 
$\log \abs{g_j}$ pour une certaine fonction analytique inversible $g_j$ sur $W_j$. 
Pour tout $(i,j)$, posons $h_{ij}=f_i|_{V_i\cap W_j}\cdot g_j|^{-1}_{V_i\cap W_j}$. 
On a alors pour tout $(i,j)$ l'égalité 
$(f-g)|_{V_i\cap W_j}=\log\abs{h_{ij}}$, ce qui montre que  $f-g$ appartient à $\PL_X^{\mathrm{str},\Z}(V)$.
Ainsi, $\PL_X^{\mathrm{str},\Z}$ est bien un sous-faisceau en groupes abéliens 
de $\PL_X|_{X_{\groth^{\mathrm{str}}}}$. 

Pour tout $(i,j)$, notons $V_{ij}^-$ et $V_{ij}^+$ les domaines 
analytiques stricts 
de~$V_i\cap W_j$ définis par les inégalités $\abs{f_i/g_i}\leq 1$
et $\abs{g_i/f_i}\leq 1$ respectivement. Les $V_{ij}^+$ et $V_{ij}^-$ constituent un 
\Gstr -recouvrement de $V$. 
Sur $X_i^+$, on a $\max(f,g)=\log\abs{f_i}$ et $\min(f,g)=\log\abs{g_i}$;
sur $X_i^-$, on a $\max(f,g)=\log\abs{g_i}$ et $\min(f,g)=\log\abs{f_i}$.
Il en résulte que $\max(f,g)$ et $\min(f,g)$ appartiennent à $\PL_X^{\mathrm{str},\Z}(V)$. 
\end{proof}

\begin{prop}\label{prop.pl-lisse-berk}
Soit $X$ un espace $k$-analytique topologiquement séparé, 
soit $x\in X$ et et soit $u$ une fonction paralinéaire et lisse sur $X$. 
Il existe alors un entier~$n\geq 1$, un voisinage analytique 
$U$ de $x$ dans $X$, un nombre réel~$c$
et une fonction holomorphe inversible~$h$ 
sur $U$ telle que $n u|_U=c+\log(\abs h)$.

Si de plus la valeur absolue de~$k$ n'est pas triviale
et que le germe~$(X,x)$ est strict, on peut même supposer $c=0$. 
\end{prop}
\begin{proof}
%
La fonction~$u$ est lisse.
On peut donc supposer que $X$ est un espace analytique compact
et qu'il existe un moment $f\colon X\to \gm^m$ 
et une fonction lisse~$v$ sur~$f_\trop(X)$ telle que $u=f^*v$. 
Comme $u$ est paalinéaire, la G-forme $\ddc u$ sur~$X$ est nulle. 
Par suite, la G-forme $\ddc v$ sur~$f_\trop(X)$ est nulle.
Cela prouve que $v$ est une fonction paralinéaire sur~$f_\trop(X)$.

D'après la proposition~\ref{prop.pl-lisse-polytope}, il existe 
un voisinage~$P$ de~$f_\trop(x)$ dans~$f_\trop(X)$
et une fonction affine~$w$ sur~$\R^m$, à pentes rationnelles,
telle que $v|_P=w|_P$. La restriction de~$u$ et de~$f^*w$
au voisinage~$f^{-1}_\trop(P)$ de~$x$ coïncident.
Soit $n$ un entier tel que $nw$ soit à pentes entières;
on écrit $nw=c+\sum a_i x_i$, avec $a_i\in\Z$ et $c\in\R$.
Sur $f_\trop^{-1}(P)$, on a $nu=c+\log \abs{\prod f_i^{a_i}}$.
Ceci entraîne la première partie de la proposition.

Supposons maintenant que $X$ est strict et que la valeur absolue de~$k$
n'est pas triviale. Les points rigides de~$X$ sont alors denses dans~$X$.
En évaluant la relation $nu=c+\sum a_i \log\abs{f_i}$
en un tel point rigide~$y$ de $f_\trop^{-1}(P)$,
on obtient $nu(y)=c+ \sum a_i \log\abs{f_{i}(y)}$.
Comme $u$ est paralinéaire, on a $u(y)\in \log(\abs{\mathscr H(y)}^\times)
\subset \Q\log\abs{k^\times}$ car $y$ est rigide;
la même propriété vaut pour les $\log\abs{f_i(y)}$.
Par conséquent, il existe un entier non nul~$m$ et un élément non nul~$a$
de~$k$ tel que $mc=\log\abs a$. On a alors
$ nmu \equiv \log \abs{ a \prod f_i^{ma_i}}$ sur $f_\trop^{-1}(P)$.
\end{proof}
\begin{rema}
L'exemple suivant montre 
qu'on ne peut pas omettre  l'hypothèse que le germe~$(X,x)$
est strict dans l'énoncé de la proposition~\ref{prop.pl-lisse-berk}.

Soit $r\in\R_+^\times$ un nombre réel qui n'est  pas de torsion
modulo~$\abs {k^\times}$.
Dans le bidisque $\{\abs{T_1}\leq r, \abs{T_2}\leq r\}$
du plan affine~$\A^2$, considérons 
la réunion~$X$ des deux domaines analytiques $X_1=\{\abs{T_1}=r\}$
et $X_2=\{\abs{T_2}=r\}$. 
Le point de Gauss~$x=\eta_{r,r}$ appartient à~$X$.

La fonction constante de valeur~$\log(r)$ sur~$X$ est paralinéaire, 
parce qu'elle est égale à~$\log(\abs {T_1})$ sur~$X_1$
et à~$\log(\abs{T_2})$ sur~$X_2$. Démontrons par l'absurde qu'il n'existe pas
de fonction $f\in\mathscr O_{X,x}^\times$ telle que $\abs f=r$ au voisinage de $x$ ; 
on suppose donc qu'il existe une telle $f$. 

Comme $r$ n'appartient pas à~$\abs{k^\times}^\Q$,
la première projection identifie l'espace~$X_1$ 
au disque $\{\abs{T_2}\leq r\} $ sur le corps~$k_r=k\{T_1/r,r/T_1\}$
des séries $ \sum a_n T_1^n$ telles que $a_n r^n\to0$ quand $\abs n\to+\infty$.
Le corps résiduel $\hrt x$gradué s'identifie alors à~$\widetilde {k_r}(\tau)$,
où $\tau$ est la réduction de~$(T_2/T_1)(x)$ (qui est de degré $1$). 

Posons $g=f/T_1$.
Par hypothèse, la valeur absolue de la fonction~$g$
est constante, égale à~$1$, au voisinage de~$x$ dans~$X_1$.
Cela entraîne que sa réduction $\widetilde{g(x)}$
est un élément de~$\widetilde {k_r}_(\tau)$ inversible sur $\Spec(\widetilde {k_r}[\tau])$; c'est donc une constante non nulle.
Il existe alors $a\in k_r^\times$ tel que $\abs a=1$
et tel que $g-a$ soit de norme~$<1$ au voisinage de~$x$.
Écrivons $a=\sum a_n T_1^n$; comme $r\not\in \abs{k^\times}^\Q$,
on a $\abs{a_0}=1$ et $\abs{a_n}r^n<1$ pour tout $n\neq0$;
en remplaçant~$a$ par~$a_0$, on se ramène au cas où $a\in k^\times$.

On a ainsi démontré qu'il existe $a\in k^\times$ tel que $\abs a=1$
et $\abs{f-aT_1}<r$ au voisinage de~$x$ dans~$X_1$.
Soit $s$ un nombre réel tel que $\abs{f-aT_1}(x)<s<r$.
Au voisinage de~$x$ dans~$X$, on a donc $\abs{f-aT_1}\leq s$.
Cela entraîne que $\abs f=\abs{T_1}$ au voisinage de~$x$ dans~$X$,
et n'est donc pas constante de valeur~$r$
car $\abs{T_1}$ prend des valeurs~$<r$
sur tout voisinage de~$x$ dans~$X_2$.
\end{rema}

\begin{lemm}\label{lemm.separation-pl}
Supposons que la valeur absolue du corps~$k$ n'est pas triviale.
Soit~$X$ un espace analytique topologiquement séparé et
localement holomorphiquement séparé.
Soit~$x$ un point de~$X$ et soit~$V$ un voisinage de~$x$.
Il existe une fonction paralinéaire sur~$X$, que l'on
peut prendre stricte si $X$ est strict, 
qui est non nulle en~$x$
et à support  contenu dans~$V$.
\end{lemm}
\begin{proof}
Appliquons le lemme~\ref{lemm.moment-bord} à l'espace~$\mathring V$
et au point~$x$;
il existe alors un voisinage compact~$W$ de~$x$ contenu dans~$V$ 
et holomorphiquement séparé, et un moment~$f\colon W\to T$ sur~$W$
tel que $f_\trop(x)\not\in f_\trop(\partial(W/X))$.
Si $X$ est strict, on suppose également que $W$ est strict.
Considérons un cube~$Q \subset T_\trop$ 
disjoint de~$f_\trop(\partial(W/X))$
et dont l'intérieur contient le point $f_\trop(x)$;
on suppose que son arête appartient à~$\Q\log\abs{k^\times}$.

Soit  $a,b$ deux éléments de~$\Q\log\abs{k^\times}$ tels que $a<b$;
définissons une fonction $u$ sur~$\R$ par $u(t)=0$ pour $t\leq a$,
$u(t)=t-a$ pour $a\leq t\leq (a+b)/2$, 
$u(t)=b-t$ pour $(a+b)/2\leq t\leq b$
et $u(t)=0$ pour $t\geq b$. Elle est paralinéaire et les scalaires
qui interviennent dans sa définition sur dans~$\Q\log\abs{k^\times}$;
elle  est nulle hors de l'intervalle~$[a;b]$ et strictement
positive sur l'intervalle ouvert~$\mathopen]a;b\mathclose[$.
En écrivant le cube~$Q$ comme un produit d'intervalles,
on construit une fonction paralinéaire~$u$
sur~$T_\trop$, à scalaires dans~$\Q\log\abs{k^\times}$,
qui est nulle hors de~$Q$ et strictement positive sur l'intérieur de~$Q$.

La fonction $f_\trop^*u$ est paralinéaire sur~$W$, et paralinéaire stricte
si $W$ est strict ; et elle est
nulle au voisinage de~$\partial(W/X)$. Son prolongement par~$0$
à~$X$ tout entier est encore paralinéaire, et paralinéaire strict si $X$ est strict ; 
et ce prolngement répond aux conditions de l'énoncé.
\end{proof}

\begin{theo}[Stone--Weierstraß paralinéaire]\label{theo.sw-pl}
Supposons que la valeur absolue de~$k$ n'est pas triviale.
Soit $X$ un espace $k$-analytique (\resp strictement $k$-analytique) 
topologiquement séparé et localement holomorphiquement séparé.
Soit $U$ un ouvert de~$X$.
Soit $f$ une fonction continue à valeurs réelles sur~$X$
dont le support est compact et contenu dans~$U$.
Pour tout~$\eps>0$, il existe une fonction paralinéaire (\resp paralinéaire stricte)~$g$ sur~$X$
(strictement~$\Q$-\pl)
dont le support est compact, contenu dans~$U$, et telle
que $\abs{f(x)-g(x)}<\eps$ pour tout $x\in X$.
\end{theo}
\begin{proof}
%
%
Si $U$ est compact, on pose $U'=U$;
sinon, on note~$U'$ son compactifié d'Alexandroff.
Par construction, $U'$ est un espace topologique compact.
Soit~$H$ l'ensemble des fonctions paralinéaires  (\resp paralinéaires strictes) sur~$U$
qui prennent une valeur constante dans~$\Q\log\abs{k^\times}$
hors d'une partie compacte de~$U$;
elles s'identifient à des fonctions continues sur~$U'$.
Pour $u,v\in H$, les fonctions $\min (u,v)$ et $\max(u,v)$
appartiennent encore à $H$
(lemme~\ref{lemm.pl-inf-sup}).

Soit $x,y$ des points de~$U'$ tels que $x\neq y$,
soit $a,b\in\R$ et soit $\eps>0$.
Démontrons qu'il existe $v\in H$ 
tel que $\abs{v(x)-a}<\eps$ et $\abs{v(y)-b}<\eps$.

Par symétrie, on peut supposer que $x\in U$.
Soit $V$ un voisinage compact de~$x$ qui est contenu dans~$U$.
D'après le lemme~\ref{lemm.separation-pl} appliqué  
il existe une fonction paralinéaire (\resp paralinéaire stricte) $u$
sur~$X$ qui est strictement positive en~$x$ et nulle hors de~$V$.
En particulier, son support est compact, contenu dans~$U$,
et l'on a $u(y)=0$.
Choisissons $s\in\Q\log\abs{k^\times}$ tel que $\abs{s-b}<\eps$,
puis un nombre rationnel~$r$ tel que $\abs{r u(x)+s-a}<\eps$.
La fonction $v=ru+s$ est paraliénaire (strictement~paralinéaire stricte) sur~$X$,
elle vérifie $\abs{v(x)-a}<\eps$ et $\abs{v(y)-b}<\eps$.

Selon la proposition~2 de Bourbaki, Topologie générale~X, p.~34,
ceci entraîne
que $H$ est dense dans l'espace des fonctions numériques continues sur~$U'$.

Soit $f$ une fonction continue sur~$X$ dont le support est compact
et contenu dans~$U$. Soit~$\eps>0$.
D'après ce qui précède, il existe une fonction~$g\in H$
telle que $\abs{f(x)-g(x)}<\eps/2$ pour tout $x\in U$.

Si $U$ est compact, on prolonge~$g$ par~$0$ sur l'ouvert $X\setminus U$.
On obtient une fonction paralinéaire (\resp strictement paralinéaire)
dont le support est compact et contenu dans~$U$,
qui approche~$f$ à~$\eps$ près.

Supposons que $U$ n'est pas compact. Par définition de l'ensemble~$H$,
il existe $a\in\Q\log\abs{k^\times}$ tel que $g'=g-a$ 
soit une fonction paralinéaire (\resp paralinéaire stricte)
à support compact dans~$U$.
Comme $f$ est à support compact dans~$U$, on a $\abs{a}\leq \eps/2$.
Soit $x\in X$.
Si $x\in U$, 
on a $\abs{f(x)-g'(x)}\leq \abs{f(x)-g(x)}+\abs{a}\leq \eps$.
Sinon, on a $g'(x)=0=f(x)$. Cela conclut la démonstration de la proposition.
\end{proof}

\begin{coro}\label{coro.sw-pl}
On suppose que la valeur absolue de $k$ n'est pas triviale. 
Soit~$X$ un espace $k$-analytique (\resp strictement $k$-analytique)
topologiquement séparé et localement
holomorphiquement séparé.
\begin{enumerate}
\item L'ensemble des fonctions paralinéaires (\resp paralinéaires strictes)
à support compact sur~$X$
est dense dans l'espace~$\mathscr C_\cpct(X)$
des fonctions continues à support compact sur~$X$, 
muni de la topologie limite inductive des espaces
de fonctions à support dans un compact donné.
\item L'ensemble des fonctions paralinéaires (\resp strictement paralinéaires)
à support compact sur~$X$
est dense dans l'espace~$\mathscr C(X)$ des fonctions continues sur~$X$
muni de la topologie de la convergence uniforme sur tout compact.
\end{enumerate}
\end{coro}
La démonstration se déduit du théorème~\ref{theo.sw-pl}
de la même façon que le corollaire~\ref{coro.app-sw} se déduisait
du théorème~\ref{theo.app-sw},
\emph{mutatis mutandis}.

\section{Polyèdre caractéristique d'un moment}
\label{ss.polyedre-caracteristique}

Si $P$ un espace paralinéaire et $n$ un entier,
rappelons que l'on note $P^{(n)}$ l'ensemble des points $x\in P$
tels que $\dim_{x}(P)\geq n$; c'est un sous-espace paralinéaire 
fermé de~$P$.

Les paragraphes qui suivent visent à construire, pour tout
espace analytique topologiquement séparé~$X$, 
une notion de partie squelettique.
D'après la construction générale, il va suffire de définir
la notion de partie squelettique élémentaire et
le procédé que nous proposons suppose donné un entier naturel~$n$
tel que $\dim(X)\leq n$. Même si la théorie ne voit pas
les composantes irréductibles de~$X$ de dimension~$<n$,
il est commode, pour des raisons de rédaction,
de ne pas supposer que $X$ est équidimensionnel,
ni que $n$ est égal à~$\dim(X)$.

Si $n$ est un entier naturel,
l'expression « Soit $X$ un espace analytique de niveau~$n$  »
signifiera ainsi que $X$ est un espace analytique 
topologiquement séparé tel que $\dim(X)\leq n$,
et que l'on garde trace de cet entier~$n$.

Lorsque $L$ est une extension complète de~$k$,
on considérera implicitement l'espace $L$-analytique~$X_L$ 
obtenu par changement
de base comme un espace analytique de niveau~$n$.
(Il est topologiquement séparé car le morphisme de changement
de base $X_L\to X$ est topologiquement séparé.)

\begin{defi}
Soit $X$ un espace $k$-analytique de niveau~$n$.
Soit $T$ un tore et soit $f\colon X\to T$ un moment sur~$X$.
On appelle \emph{polyèdre caractéristique}\index{polyèdre caractéristique d'un moment}
du moment~$f$
la réunion des images réciproques par $q\circ f$ 
du squelette canonique de~$\gm^n$,
lorsque $q$ parcourt l'ensemble des morphismes de tores de~$T$ dans~$\gm^n$.
On le note~$\Sigma_f$.
\end{defi}

Soit $x\in\Sigma_f$.  Soit $q\colon T\to \gm^n$ un morphisme
de tores tel que $q\circ f(x)\in S(\gm^n)$. 
On a donc $d_k(x)\geq d_k(f(x)) \geq d_k(q\circ f(x))=n=\dim(X)$,
donc $d_k(x)=\dim(X)= d_k(f(x))=n$.
Par suite, le point~$x$ n'appartient à
aucun fermé de Zariski
de dimension~$<\dim(X)$ d'aucun domaine analytique de~$X$.
En particulier, le point~$x$ n'appartient qu'à une seule composante
irréductible de~$X$ dans laquelle il est Zariski-dense.

Cela démontre aussi que $\Sigma_f$ est vide si $\dim(T)<n$.

Notons $f_\red\colon X_\red\to T$  le moment déduit de~$f$ par restriction.
On a $\Sigma_{f_\red}=\Sigma_f$.

\begin{rema}
Soit $X$ un espace $k$-analytique de niveau~$n$
et soit $f\colon X\to T$ un moment sur~$X$. 
Le polyèdre caractéristique~$\Sigma_f$ est contenu dans la réunion~$X^{(n)}$
des composantes irrédutibles de~$X$ qui sont de dimension~$n$;
en particulier, cette théorie ignore les éventuelles composantes
de dimension~$<n$. On pourrait travailler sur chacune d'elle en changeant
l'entier~$n$ de référence.
\end{rema}

\begin{prop}\label{prop.Sigmaf-reunion-finie}
Soit $X$ un espace $k$-analytique de niveau~$n$.
Le polyèdre caractéristique de~$f$ est la réunion
des images réciproques par $q_I\circ f$ 
du squelette canonique de~$\gm^{n}$,
où $I$ parcourt l'ensemble des parties 
de~$\{1,\dots,m\}$ de cardinal~${n}$.
\end{prop}
\begin{proof}
Il résulte de la définition que pour toute partie~$I$ de~$\{1,\dots,m\}$
de cardinal~$n$, on a $(q_I\circ f)^{-1}(S(\gm^n))\subset\Sigma_f$;
cela démontre une inclusion.
Inversement, soit $q\colon \gm^m\to \gm^n$ un morphisme de tores
et soit $x\in X$ un point tel que $q(f(x))\in S(\gm^n)$.
Posons $t=f(x)$; comme $\gm^m$ est sans bord,
on a donc $\dim_{\trop,t}(q)=n$ (prop.~\ref{prop.trop.local}, (iv)).
Le morphisme~$q$ est formé de $n$~monômes en les caractères
canoniques~$\chi_1,\dots,\chi_m$ de~$\gm^m$. Si $\dim_{\trop,t}(q)=n$,
il existe donc une partie~$I$ de~$\{1,\dots,m\}$ de cardinal~$n$
telle que $\dim_{\trop,t}((\chi_i)_{i\in I})$ soit égale à~$n$.
On a donc $q_I(t)\in S(\gm^n)$ (prop.~\ref{prop.trop.local}, (iv)),
d'où $q_I\circ f(x)\in S(\gm^n)$.
\end{proof}

\begin{rema}\label{rema.prolongement-moment}
Soit $X$ un espace $k$-analytique de niveau~$n$.
Soit $x\in X$ un point tel que $d_k(x)=n$  et tel que le germe $(X,x)$ est bon.
Alors, $x$ appartient au polyèdre caractéristique
d'un moment défini au voisinage de~$x$.

Plus précisément,
soit $f_1,\dots,f_d$ des germes de fonctions analytiques inversibles 
telles que les $\widetilde{f_j(x)} $ soient algébriquement
indépendantes sur
le corps gradué $\widetilde k$. 
Comme le germe~$(X,x)$ est bon, il existe des germes
de fonctions analytiques inversibles $f_{d+1},\dots,f_n$
telles les $\widetilde{f_j(x)}$ (pour $1\leq j\leq n$)
forment une base de transcendance graduée de $\hrt x$ sur~$\widetilde k$.
Cela signifie que $x\in\Sigma_f$.

L'assertion initiale s'obtient en partant de la famille vide $(f_1,\dots,f_d)$
pour $d=0$.
\end{rema}

\begin{prop}\label{prop.moment-bonne-proj}
Soit $X$ un espace $k$-analytique de niveau~$n$
et soit $f\colon X\to\gm^m$ un moment.
Soit~$x$ un point de~$X$ tel que $\dim_x(X)=n$.
Il existe un morphisme de tores $p\colon \gm^m\to\gm^n$
tel que $\Sigma_f$ et $\Sigma_{p\circ f}$ coïncident au voisinage de~$x$.
\end{prop}

\begin{proof}
On a toujours $\Sigma_{p\circ f}\subset\Sigma_f$; l'assertion est donc
évidente si $x\not\in\Sigma_f$.
On suppose désormais que $x\in\Sigma_f$ et on 
raisonne par récurrence sur la longueur~$m$ de~$f$.
Elle est au moins égale à~$n$, puisque $x\in\Sigma_f$ ;
si elle est égale à~$n$, on prend $p=\id$.

Supposons $m>n$; on écrit $f=(f_1,\dots,f_m)$.
Posons $r_i=\abs{f_i(x)}$. 
Comme le degré de transcendance de~$\hrt x$ sur~$\tilde k$
est égal à~$n$,
les éléments $a_i=\widetilde {f_i(x)}$ de~$\hrt x$
sont algébriquement dépendants et il existe
donc un élément homogène non nul $F$
de l'anneau gradué $\hrt k[T_1/r_1,\dots,T_m/r_m]$
tel que $F(a_1,\dots,a_m)=0$.  
Nous allons effectuer une « normalisation de Noether torique »,
reprise 
de~\cite[preuve de la proposition~3.2.7, p.~105]{MaclaganSturmfels-2015}. 
Pour $q\in\N$ assez grand, on effectue un changement de variables
monomial en $(T_1,\dots,T_m)$ par :
\[ \phi^*(T_1)=T_1 T_m^q, \dots, \phi^*(T_{m-1})=T_{m-1} T_m^{q^{m-1}},
\phi^*(T_m)=T_m. \]
Si $q$ est assez grand, on déduit de l'unicité du développement d'un entier
en base~$q$ la propriété suivante sur le polynôme $G = \phi^*(F)$ :
ses monômes ont des degrés deux à deux distincts en la variable~$T_m$;
on écrit ainsi
\[ G = \sum_{d\in\N} c_d(T_1,\dots,T_{m-1}) T_m^d , \]
où $c_d$ est un monôme, ces monômes n'étant pas tous nuls.
Posons $b_1=a_1 a_m^{-q}$, \dots, $b_{m-1}=a_{m-1}a_m^{-q^{m-1}}$
et $b_m=a_m$; par construction, on a $G(b_1,\dots,b_m)=0$:
\[ \sum_{d\in\N} c_d(b_1,\dots,b_{m-1}) b_m^d = 0. \]

Posons $g_i = f_i f_m^{-q^i}$ pour $1\leq i< m$ et $g_m=f_m$;
ce sont des fonctions inversibles sur~$X$ et
on a $b_i=\widetilde{g_i(x)}$ par construction.
Comme le changement de fonctions des~$f_i$ aux~$g_i$ est monomial inversible,
on a $\Sigma_f=\Sigma_g$.

Posons $g'=(g_1,\dots,g_{m-1})$; il est de la forme $p\circ f$,
pour $p$ convenable et
nous allons démontrer que l'on a
l'égalité $\Sigma_g=\Sigma_{g'}$ au voisinage de~$x$,
d'où l'assertion voulue par récurrence sur~$m$.

On a 
\[ \sum_{d\in\N} c_d(\widetilde{g_1(x)},\dots,\widetilde{g_{m-1}(x)})
 \widetilde{g_m(x)}^d = 0, \]
ce qui signifie que 
\[ \abs{\sum_d c_d (g_1(x),\dots,g_{m-1}(x)) g_m(x)^d}
 < \sup_d \abs{c_d (g_1(x),\dots,g_{m-1}(x))} \cdot \abs{g_m(x)}^d. \]
L'ouvert $W$ de $X$ constitué des points $y$ els que 
\[ \abs{\sum_d c_d (g_1(y),\dots,g_{m-1}(y)) g_m(y)^d}
 < \sup_d \abs{c_d (g_1(y),\dots,g_{m-1}(y))} \cdot \abs{g_m(y)}^d \]
contient par conséquent $x$. 

Soit $y\in W$. On a une relation 
\[ \sum_{d\in E} c_d(\widetilde{g_1(y)},\dots,\widetilde{g_{m-1}(y)}) \widetilde{g_m(y)}^d = 0 \]
où $E$ est l'ensemble non vide des entiers $d\in\N$ tels que
$\abs{c_d (g_1(y),\dots,g_{m-1}(y))}\cdot  \abs{g_m(y)}^d$ est maximal.
En particulier, $\widetilde{g_m(y)}$ est algébrique sur
$\tilde k(\widetilde{g_1(y)},\dots,\widetilde{g_{m-1}(y)})$,
si bien que $y\in\Sigma_g$ si et seulement si 
$y\in\Sigma_{g'}$.
Cela conclut la démonstration.
\end{proof}

\begin{lemm}\label{lemme-sigmaf-paralineaire}
Soit $X$ un espace $k$-analytique de niveau~$n$
et soit $f\colon X\to \gm^n$ un moment sur $X$. 
Le sous-ensemble $\Sigma_f$ de $X$ en est alors une partie paralinéaire
et l'application $f_\trop|_{\Sigma_f}\colon \Sigma_f\to\R^n$
est paralinéaire.
\end{lemm}
\begin{proof}
L'assertion est G-locale sur $X$ d'après le lemme
\ref{lemme-partiespl-glocal} ; on peut donc supposer $X$
affinoïde. 
Par ailleurs, l'assertion est insensible à la présence de nilpotents ; 
on peut donc remplacer $X$ par $X_{\mathrm{red}}$ et ainsi le supposer réduit. 
Soit $x\in \Sigma_f$. 
Toujours d'après le lemme \ref{lemme-partiespl-glocal}, il suffit d'exhiber
un voisinage affinoïde $V$ de $x$ dans $X$ tel que $V\cap  \Sigma_f$ soit
une partie paralinéaire de $V$ en restriction à laquelle $f_\trop$
est paralinéaire. Or cela résulte de la preuve du théorème 5.1
de \cite{ducros2012b} (et notamment du paragraphe 5.5.1 de celle-ci) 
qui montre même que $V$ peut être choisi de sorte que $\Sigma_f\cap V$ soit
\textit{fidèlement} paralinéaire. 
\end{proof}

\begin{prop}\label{prop-unionsigmaf-paralineaire}
Soit $X$ un espace $k$-analytique de niveau~$n$, 
soit $(X_i)$
une famille localement finie de domaines analytiques fermés;
pour tout $i$, soit $f_i \colon X_i\to \gm^{n_i}$ un moment. 
Posons $\Sigma=\bigcup_i \Sigma_{f_i}$. 
Il existe une famille
$(W_j)$ de domaines analytiques de~$X$
qui G-recouvrent $\Sigma$ et pour tout $j$, un moment 
$g_j\colon W_j\to \gm^{n}$ tel que $\Sigma\cap W_j$ soit une partie
paralinéaire de $\Sigma_{g_j}$. 
En particulier, $\Sigma$ est paralinéaire.
\end{prop}
\begin{proof}
L'assertion est G-locale sur $X$, 
ce qui permet de supposer que 
$X$ est affinoïde et que $(X_i)$ est fini. 
Pour tout~$i$, notons $(f_{i,j})$ les composantes de~$f_i$.
Soit $x$ un point de $\Sigma$.
Soit $I$ l'ensemble des indices tels que $x\in \Sigma_{f_i}$. 
Comme $x$ est situé sur au moins un $\Sigma_{f_i}$, on a $n=d_k(x)$, 
et il existe donc un voisinage affinoïde $V_x$ de $x$ dans $X$ 
qui est purement de dimension~$n$. 
De plus, quitte à restreindre $V_x$, on peut supposer que les
propriétés suivantes soient satisfaites : 
\begin{itemize}
\item $V_x\cap \Sigma_i=\emptyset$ dès que $i\notin I$ ; 
\item pour tout $i\in I$, et pour tout~$j$,
on a $\abs{f_{i,j}}\geq \abs{f_{i,j}(x)}/2$ sur $X_i\cap V_x$ ; 
\item pour tout $i\in I$ et tout~$j$,
il existe une fonction analytique inversible $f'_{i,j}$ sur $V_x$ 
telle que $\abs{f_{i,j}-f'_{i,j}}<\abs{f_{i,j}(x)}/2$ sur $X_i\cap V_x$ 
\end{itemize} 
(la troisième condition est réalisable car $X$ est affinoïde, 
donc bon, et en particulier presque bon en $x$). 
Pour tout~$i$, notons~$f'_i$ le moment  de composantes $(f'_{i,j})$.
Par construction, on a alors $\Sigma_{f_i}\cap V_x=\Sigma_{f'_i}\cap X_i$ pour tout $i\in I$. 
Soit $f$ la concaténation des $f'_i$. 
La proposition \ref{prop.moment-bonne-proj} assure l'existence d'un moment $g\colon V_x
\to \gm^n$ tel que $\Sigma_f\subset \Sigma_g$. En particulier, chacun des $\Sigma_{f_i}\cap V_x$ est contenu
dans $\Sigma_g$, et est par ailleurs
une partie paralinéaire de $V_x\cap X_i$ 
par le lemme \ref{lemme-sigmaf-paralineaire} ; 
c'est donc une partie paralinéaire 
de $\Sigma_g$. Et comme $\Sigma \cap V_x$ est la réunion des
$\Sigma_{f_i}\cap V_x$, c'est encore une partie paralinéaire de $\Sigma_g$.

Il découle alors de la proposition~\ref{lemm.paralineaire-Glocal} que $\Sigma$
est paralinéaire.
\end{proof}

\begin{coro}\label{coro.polcar-pl}
Soit $X$ un espace $k$-analytique  de niveau~$n$.
Le polyèdre caractéristique d'un moment $f\colon X\to \gm^m$ 
est une partie paralinéaire de~$X$,
et l'application $f_\trop|_{\Sigma_f}\colon \Sigma_f\to\R^m$  
est paralinéaire.
\end{coro}
\begin{proof}
Compte tenu de la proposition~\ref{prop.Sigmaf-reunion-finie},
on écrit $\Sigma_f$ comme la réunion des
$\Sigma_{q_I\circ f}$, où $I$ parcourt l'ensemble
des parties de~$\{1,\dots,m\}$ de cardinal~$n$.
Il résulte alors de la proposition précédente,
appliquée à la famille des moments $(q_I\circ f)$, 
que $\Sigma_f$ est une partie paralinéaire de~$X$.

Par définition, $f_\trop$ est une application paralinéaire sur~$X$
et donc sa restriction à toute partie paralinéaire est paralinéaire.
D'après le lemme~\ref{lemme-sigmaf-paralineaire}, 
la restriction à $\Sigma_{q_I\circ f}$ de~$(q_I\circ f)_\trop$ est une immersion par morceaux ; par suite, la
restriction de~$f_\trop$ est une immersion par morceaux.
Comme la famille des $\Sigma_{q_I\circ f}$ est un G-recouvrement de~$\Sigma$,
la restriction à~$\Sigma$ de~$f_\trop$ est une immersion par morceaux.
\end{proof}

\section{Polyèdres caractéristiques et changement de base}

\subsection{}
Soit $X$ un espace $k$-analytique de niveau~$n$.
Soit $f\colon X\to T$ un moment et soit $L$ une extension complète de~$k$.
On note $p\colon X_L\to X$ le morphisme de changement de base
et $f_L\colon X_L\to T_L$
le moment déduit par changement de base. 
On voudrait décrire le lien entre~$\Sigma_{f_L}$ et~$\Sigma_f$.

Dans les deux cas « extrêmes » qui correspondent
essentiellement au cas d'une extension résiduelle graduée qui 
est soit algébrique soit transcendente pure, les propositions~\ref{prop.polyedre-chgt-base-algebrique} et le corollaire~\ref{coro.Sigma-chgt-base}
indiquent que le polyèdre
caractéristique se comporte de façon naturelle.

\begin{prop}\label{prop.polyedre-chgt-base-algebrique}
Supposons que le corps résiduel gradué~$\tilde L$ soit algébrique
sur~$\tilde k$; c'est par exemple le cas si $L$ contient un sous-corps
dense qui est algébrique sur~$k$, ou si $L$ est une extension immédiate de~$k$.
Alors, on a $\Sigma_{f_L}=p^{-1}(\Sigma_f)$.
\end{prop}
\begin{proof}
En revenant la définition du polyèdre caractéristique
il suffit, grâce à la proposition~\ref{prop.Sigmaf-reunion-finie},
de traiter le cas où $X$ est irréductible, de dimension~$d$,
et où $T=\gm^d$.

Soit $y\in X_L$ et soit $x=p(y)$.
Pour que $x\in\Sigma_f$, il faut et il suffit que
la famille $(\widetilde{f_1(x)},\dots,\widetilde{f_d(x)})$
soit algébriquement indépendante sur~$\tilde k$;
pour que $y\in\Sigma_{f_L}$, il faut et il suffit que
la famille $(\widetilde{f_1(x)},\dots,\widetilde{f_d(x)})$
soit algébriquement indépendante sur~$\tilde L$;
l'assertion en résulte.
\end{proof}

\subsection{}
Soit $\mathscr A$ un anneau de Banach. 
Son spectre $\mathscr M(\mathscr A)$ est l'ensemble 
des semi-normes multiplicatives bornées sur~$\mathscr A$.
Il est muni d'une relation d'ordre pour laquelle $x\preceq y$
si $\abs{f(x)}\leq \abs{f(y)}$ pour tout $f\in\mathscr M(\mathscr A)$.
Le plus grand élément de~$\mathscr M(\mathscr A)$, s'il existe,
est donc un élément~$x$ tel que $\abs{f(x)}\geq \abs{f(y)}$ pour tout $y\in\mathscr M(\mathscr A)$.
Si la norme de~$\mathscr A$ est elle-même multiplicative,
c'est le le plus  grand élément de~$\mathscr M(\mathscr A)$.

\begin{lemm}
Soit $x\in\Sigma_f$.
Supposons que $p^{-1}(x)=\mathscr M(\hr x \mathop{\widehat\otimes} L)$
possède un plus grand élément, $y$.
Alors $y$ est l'unique antécédent de~$x$ sur~$\Sigma_{f_L}$.

Supposons que la norme tensorielle de $\hr x \mathop{\widehat\otimes} L$ 
soit multiplicative et notons~$y$ le point de~$X_L$ au-dessus de~$x$
qui correspond à cette norme.
Alors, le point~$y$ est l'unique antécédent de~$x$ dans~$\Sigma_{f_L}$.
\end{lemm}
\begin{proof}
On se ramène encore au cas où $X$ est irréductible, de dimension~$d$,
et où $f=(f_1,\dots,f_d)\colon X\to\gm^d$.
Posons $t=f(x)$ et $s=f_L(y)$. Notons~$q$ le morphisme canonique
de~$T_L$ sur~$T$.
Par changement de base, on a $p^{-1}(x)= q^{-1}(t) \otimes_{\hr t} \hr x$,
de sorte que le morphisme canonique $p^{-1}(x)\to q^{-1}(t)$ est surjectif,
et, le point~$y$ étant le plus grand élément de~$p^{-1}(x)$, 
le point~$s$ est ainsi le plus grand élément de $q^{-1}(t)$.
Soit $r\in (\R_+^\times)^d$ tel que $t=\eta_r$; on a donc $s=\eta_{L,r}$.
Puisque $f_L(y)=s$, cela prouve que $y\in\Sigma_{f_L}$.

Raisonnons par l'absurde et considérons
un antécédent~$y'$ de~$x$ dans~$\Sigma_{f_L}$ tel que $y'\neq y$.
Notons $\mathscr A=\mathscr O(q^{-1}(t))$, $\mathscr B=\mathscr O(p^{-1}(s))$,
munissons ces algèbres de leurs normes tensorielles qui en font
des algèbres de Banach. 
L'algèbre~$\mathscr A$ est intègre, car sa norme est multiplicative
(\cite[lemma~5.2.2]{berkovich1990}; 
\cite[lemme~1.8 et exemple~1.9]{ducros2009}).
L'algèbre $\mathscr B$ est finie sur~$\mathscr A$ car 
$f^{-1}(t)$ est de dimension nulle.
Il existe une fonction $\phi\in\mathscr O(f_L^{-1}(s)\cap p^{-1}(x))$
telle que $\abs{\phi(y')}=1$ et $\abs{\phi(y)}=0$.
Puisque l'anneau $\mathscr O(f_L^{-1}(s)\cap p^{-1}(x))$
est le complété du produit tensoriel $\mathscr B\otimes_{\mathscr A}
\operatorname{Frac}(\mathscr A)$, il existe un élément
$\psi$ dans ce produit tensoriel tel que $\norm{\psi-\phi}<1/2$.
Alors $\abs{\psi(y')}>1/2>\abs{\psi(y)}$.
Considérons $\lambda\in\mathscr B$ et $\mu\in\mathscr A\setminus\{0\}$
tels que $\psi=\lambda/\mu$. Comme $\abs{\mu(y')}=\abs{\mu(y)}=\abs{\mu(s)}$,
on a donc $\abs{\lambda(y')}>\abs{\lambda(y)}$, ce qui contredit
l'hypothèse que $y$ est le plus grand élément de~$p^{-1}(x)$.
\end{proof}

Rappelons
(\cite[définition~3.2]{poineau2013}) qu'une extension
complète $L$ de~$k$ est dite universellement multiplicative
si pour toute extension complète~$K$ de~$k$, la norme tensorielle
de l'algèbre $K\mathop{\widehat \otimes}_kL$ est multiplicative.
Dans ce cas, \cite{poineau2013} démontre que l'application
$\sigma_L$ l'application de~$X$ dans~$X_L$ qui applique
un point $x\in X$ sur le plus grand élément de~$p^{-1}(x)$
est une section continue de la projection de~$X_L$ sur~$X$
(corollaire~3.7).

Voici un exemple important d'extension universellement multiplicative.
On considère une famille
$r=(r_1,\dots,r_n)$ de nombres réels strictement positifs;
soit $\eta_r$ le point de Shilov correspondant de~$\gm^n$, tel que $\abs{T_i(\eta_r)}=r_i$ pour tout~$i$, et l'on pose $L=\mathscr H(\eta_r)$.
Alors, l'extension $L$ de~$k$ est universellement multiplicative
et la section~$\sigma_L$ se déduit de
la « section de Shilov » $T_\trop \times X \to T\times_k X$.

\begin{coro}\label{coro.Sigma-chgt-base}
Soit $L$ une extension complète universellement multiplicative de~$k$.
On a $\Sigma_{f_L}=\sigma_L(\Sigma_f)$.
\end{coro}

Notons $X_{\mathrm u}$ l'ensemble des points universels de~$X$,
c'est-à-dire des points~$x\in X$ tels que l'extension~$\hr x$
de~$x$ soit universellement multiplicative.
C'est une partie fermée de~$X$ (\cite[proposition~3.12]{poineau2013})
et l'application $\sigma_L\colon X_{\mathrm u}\to X_L$ qui, à tout point
universel~$x$, associe le point de~$X_L$ au-dessus de~$x$
correspondant à la norme tensorielle de~$\hr x\mathop{\widehat \otimes} L$,
est une section continue au-dessus de~$X_{\mathrm u}$
de la projection canonique $X_L\to X$
(\loccit, corollaire~3.7).
Lorsque le corps~$k$ est algébriquement clos, on a $X_{\mathrm u}=X$
(\loccit, corollaire~3.14).

\begin{coro}
On $\Sigma_{f_L} \cap p^{-1}(X_{\mathrm u}) 
= \sigma_L( \Sigma_f \cap X_{\mathrm u})$.
En particulier, si $k$ est algébriquement clos, on a $\Sigma_{f_L}=\sigma_L(\Sigma_f)$.
\end{coro}

Lorsqu'on ne fait pas d'hypothèse sur l'extension~$L$ de~$k$,
nous pouvons préciser un résultat de~\cite{ducros2012b}.
\begin{prop}\label{prop.polyedre-chgt-base}
\begin{enumerate}
\item Pour toute extension complète~$L$ de~$k$,
l'application~$p$ induit une immersion par morceaux surjective
de~$\Sigma_{f_L}$ sur~$\Sigma_f$ qui  préserve la dimension locale.
\item Il existe une extension algébrique séparable~$F$ de~$k$,
de complétée~$\widehat F$, telle que,
pour toute extension complète~$L$ de~$\widehat F$, 
la projection canonique de~$\Sigma_{f_L}$ sur~$\Sigma_{f_{\widehat F}}$
est un isomorphisme d'espaces paralinéaires.
Si, de plus, $X$ est compact, on peut prendre~$F$ finie sur~$k$
(et donc $\widehat F=F$).
\item En particulier, si $k$ est algébriquement clos, 
alors la projection canonique de~$\Sigma_{f_L}$ sur~$\Sigma_{f}$
est un isomorphisme, pour toute extension complète~$L$ de~$k$.
\end{enumerate}
\end{prop}
\begin{proof}
D'après~\cite[théorème~5.1, 3)]{ducros2012b},
l'application~$p$ induit une immersion par morceaux surjective
de~$\Sigma_{f_L}$ sur~$\Sigma_F$, ce qui démontre
la première partie de l'assertion~\emph a).
Si $X$ est compact, 
le théorème~5.1, 4) de~\cite{ducros2012b} fournit l'assertion~\emph b),
l'extension~$F$ de~$k$ étant en outre finie.
Puisque la formation des polyèdres caractéristiques est G-locale,
l'assertion~\emph b) en résulte en composant toutes les extensions finies
de~$k$ fournies par un recouvrement affinoïde de~$X$.
L'assertion~\emph c) en découle immédiatement.

Prouvons maintenant que le morphisme canonique de~$\Sigma_{f_L}$
sur~$\Sigma_f$ préserve la dimension locale.
On peut raisonner G-localement sur~$X$, et supposer que $X$ est compact.
Soit $F$ une extension finie séparable de~$k$, qu'on suppose
galoisienne, stabilisant les polyèdres caractéristiques
et soit $M$ l'extension  composée de~$L$ et~$F$.
Soit $y\in \Sigma_{f_L}$, soit $x$ son image dans~$\Sigma_f$,
soit $z$ un point de~$\Sigma_{f_M}$ au-dessus de~$y$
et $t$ sa projection dans~$\Sigma_f$.
Comme la projection $\Sigma_{f_M}\to \Sigma_{f_F}$ est un isomorphisme,
on a $\dim_z(\Sigma_{f_M})=\dim_t(\Sigma_f)$.
D'après la proposition~\ref{prop.polyedre-chgt-base-algebrique},
le polyèdre caractéristique~$\Sigma_f$ est le quotient
de~$\Sigma_{f_F}$ par l'action de~$\Gal(F/k)$;
on a donc $\dim_t(\Sigma_{f_F})=\dim_x(\Sigma_f)$,
donc $\dim_z(\Sigma{f_M})=\dim_x(\Sigma_f)$.
Comme les projections $\Sigma_{f_M}\to \Sigma_{f_L}$
et $\Sigma_{f_L}\to \Sigma_f$ sont des immersions par morceaux
d'espaces paralinéaires,
on a $\dim_z(\Sigma_{f_M})\leq \dim_y(\Sigma_{f_L})\leq \dim_x(\Sigma_f)$.
Cela démontre que $\dim_y(\Sigma_{f_L})=\dim_x(\Sigma_f)$.
\end{proof}

\section{Polyèdres caractéristiques et dimension tropicale}

\subsection{}
 Soit $X$ un espace $k$-analytique de niveau~$n$, 
soit $m$ un entier et soit $f= X\to \gm^m$ un moment sur~$X$.
On note $(f_1,\dots,f_m)$ les composantes de~$f$.
Pour toute partie~$I$ de~$\{1,\dots,m\}$,
on note $q_I$ la projection canonique de~$\gm^m$ sur~$\gm^I$.

Rappelons (proposition~\ref{coro.polcar-pl}) que
le polyèdre caractéristique de~$f$ est une partie paralinéaire de~$X$ 
et que l'application $f_\trop|_{\Sigma_f}$ est une immersion par morceaux
d'espaces paralinéaires.

Le but de ce numéro est de relier la dimension tropicale
de~$f_\trop$ aux propriétés de~$\Sigma_f$.

\begin{theo}\label{theo.proprietes-Sigmaf}\phantomsection
Soit $x$ un point de $X$. 
\begin{enumerate}
\item
On a $\dim_{\trop,x}(f)\leq n$,
avec égalité 
si et seulement si $x\in\Sigma_f^{(n)}$.

\item
Supposons $x\in \Int(X)$. Les trois assertions suivantes
sont équivalentes : 
\textup{(i)} $x\in  \Sigma_f$ ; 
\textup{(ii)} $f$ est de dimension tropicale~$n$ en~$x$ ;
\textup{(iii)} $x\in \Sigma_f^{(n)}$. 

\item
Si $X$ est compact
ou plus généralement
si l'application $f_\trop$ est topologiquement
propre, 
on a $f_\trop(X)^{(n)}=f_\trop(\Sigma_f^{(n)})$.
\end{enumerate}
\end{theo}
\begin{proof}
\begin{enumerate}
\item
%
 Soit $d$ le degré de transcendance résiduel de $f$ en $x$. 
On a $\dim_{\trop,x}(f)\leq d\leq \dim_x X$ (\eqref{eq.dimtrop-dk-dim}).
Comme $\dim_x X\leq n$, il vient $\dim_{\trop,x}(f)\leq d\leq n$. 
 
 Supposons que $\dim_{\trop,x}(f)=n$, et montrons
que $x\in \Sigma_f^{(n)}$. 
Puisque $\dim_{\trop,x}(f)\leq d\leq n$, on a 
$d=n$. Il existe donc un sous-ensemble $I$ de cardinal $d$ de $\{1,\ldots, m\}$ 
tels que les $\widetilde{f_i(x)}$ soient algébriquement indépendants sur $\widetilde k$. 
ll s'ensuit que $x\in \Sigma_f$. 
Il reste à vérifier que 
$\dim_x(\Sigma_f)=n$.
Soit~$V$ un voisinage analytique compact arbitraire de~$x$ dans~$X$
et démontrons que $\dim(\Sigma_f\cap V)=n$.
Comme $\dim_{\trop,x}(f)=n$, on a $\dim(f_\trop(V))=n$.
Soit~$C$ une cellule compacte de dimension~$n$
de $f_\trop(V)$ et soit $q\colon \gm^m\to \gm^n$ un morphisme
de tores tel que la restriction de~$q_\trop$ à~$\langle C\rangle$ 
soit injective. Alors, $\dim (q\circ f)_\trop(V)=n$.
D'après le lemme~\ref{lemm.degre-positif}, 
$(q\circ f)_\trop(V)^{(n)}_\sq$
est contenu dans $(q\circ f)(V)$, donc 
dans $(q\circ f)(\Sigma_f\cap V)$. 
Comme $\dim(\Sigma_f\cap V)\geq n$,
il vient $\dim(\Sigma_f\cap V)=n$, ce qu'il fallait démontrer.

Inversement, supposons que $x\in \Sigma_f^{(n)}$ et
montrons que $\dim_{\trop,x}(f)=n$. 
Soit $V$ un voisinage compact de~$x$.
Alors, \mbox{$V\cap \Sigma_f^{(n)}$} est un voisinage de~$x$ 
dans~$\Sigma_f^{(n)}$,
donc contient une cellule de dimension~$n$ puisque $\dim_x\Sigma_f=n$ par hypothèse. 
Comme la restriction de~$f_\trop$ à~$\Sigma_f^{(n)}$ 
est une immersion par morceaux, cela prouve que $\dim(f_\trop(V))\geq n$, 
et donc que $\dim(f_\trop(V))=n$ puisque $\dim V\leq n$. 
Comme $V$ est arbitraire, il vient $\dim_{\trop,x}(f)=n$. 

\item
Faisons l'hypothèse que $x\in\Int(X)$. 
Supposons $x\in \Sigma_f$. Il existe alors un sous-ensemble $I$
de $\{1,\ldots, m\}$ tels que les $\widetilde{f_i(x)}$ pour $i\in I$ soient 
algébriquement indépendantes sur $\widetilde k$ ; par conséquent, $d\geq n$, 
si bien que $d=n$ puisqu'on sait que $d\leq n$. 
Et comme $x\in \Int (X)$, la dimension tropicale 
de $f$ en $x$ est égale à $d$ (\ref{eq.dimtrop-dk-dim}), et donc à~$n$.

Supposons réciproquement que $x$ est de dimension 
tropicale $n$ en $x$. Le point $x$ appartient alors à 
$\Sigma_f^{(n)}$ d'après~(a). 

L'implication (iii)$\Rightarrow$(i) est évidente. 

\item
Supposons 
$f_\trop$ topologiquement propre. En raisonnant G-localement sur
$\R^m$ on voit que $f_\trop(X), f_\trop(\Sigma_f)$ et $f_\trop(\Sigma_f^{(n)})$ sont
des parties paralinéaires
fermées de $\R^m$. Comme $f_\trop$ induit une immersion par morceaux de 
$\Sigma_f$ sur $f_\trop(\Sigma_f)$, l'image $f_\trop(\Sigma_f^{(n)})$ est purement de 
dimension $n$ et est donc contenue dans $f_\trop(X)^{(n)}$.

Soit maintenant $\xi$ un point de $f_\trop(X)$ qui n'appartient pas
à $f_\trop(\Sigma_f^{(n)})$. 
Soit~$x$ un point de~$X$ tel que $f_\trop(x)=\xi$; on a donc $x\not\in\Sigma_f^{(n)}$.
D'après~\emph a), $f_\trop$ est de dimension tropicale~$< n$ en~$x$.
Comme $f_\trop$ est topologiquement propre, 
il existe un voisinage analytique compact~$V$ de~$f_\trop^{-1}(\xi)$
et un voisinage paralinéaire~$U$ de~$\xi$ dans~$f_\trop(X)$ 
tels que $f_\trop^{-1}(U)\subset V$ et 
que $\dim(f_\trop(V))\leq n-1$.
Puisque $U\subset f_\trop(V)$, il est de dimension au plus~$n-1$ 
et $\xi$ n'appartient à $f_\trop(X)^{(n)}$.
\qedhere
\end{enumerate}
\end{proof}

\begin{coro}\label{coro.proprietes-Sigmaf}
Soit $X$ un espace $k$-analytique de niveau~$n$ et
soit $f\colon X\to T$ un moment. 
Alors $\Sigma_f^{(n)}$ coïncide avec le polyèdre
de variation maximale~$\Sigma_{f_\trop}^{(n)}$ de 
l'application paralinéaire~$f_\trop$ 
(définition~\ref{defi.polyedre-variation-maximale}).
\end{coro}

\begin{proof}
C'est l'assertion~\emph a) du théorème~\ref{theo.proprietes-Sigmaf}.
\end{proof}

\begin{coro}\label{coro.Sigmaf=Sigmafn}
Soit $X$ un espace $k$-analytique de niveau~$n$, sans bord.
Pour tout moment~$f\colon X\to T$, on a $\Sigma_f^{(n)}=\Sigma_f$.
\end{coro}
\begin{proof}
Comme $X$ est sans bord, on a $X=\Int(X)$.
L'assertion découle alors du théorème~\ref{theo.proprietes-Sigmaf}, \emph b).
\end{proof}

\begin{rema}
Soit $X$ un espace $k$-analytique de niveau~$n$
et soit $f$ un moment sur~$X$.
Puisqu'elle est égale à~$\Sigma_{f_\trop}^{(n)}$,
la partie paralinéaire~$\Sigma_f^{(n)}$ de~$X$ ne dépend que de~$f_\trop$. 
Ce n'est en revanche pas le cas de~$\Sigma_f$ en général
(à moins que $X$ ne soit sans bord auquel cas $\Sigma_f=\Sigma_f^{(n)}$).

Par exemple, prenons pour~$X$ le domaine $\abs{T(x)}=1$  dans~$\gm$
et pour $f$ le moment donné par l'immersion canonique.
Alors $\Sigma_f$ est réduit au point de Gauss de~$\gm$.
Sur~$X$, $f_\trop$ est la fonction constante de valeur~$1$,
donc $f_\trop=1_\trop$ mais $\Sigma_1=\emptyset$.
\end{rema}

\section{Parties squelettiques d'un espace analytique}
\label{ss.polyedre.caracteristique}

Le but de ce paragraphe est de définir la classe 
des parties squelettiques d'un espace analytique,
le munissant ainsi d'une structure d'espace G-tropical vertébré.

\subsection{}
Soit $X$ un espace $k$-analytique de niveau~$n$.
Soit $g\colon X\to\R^m$ une application paralinéaire sur~$X$.
Rappelons que l'on note $\Sigma_g^{(n)}$ le lieu des points~$x$
en lesquel $g$ est de dimension tropicale~$n$.
Lorsque $f\colon X\to \gm^m$ est un moment sur~$X$ tel que $g=f_\trop$,
le corollaire~\ref{coro.proprietes-Sigmaf} affirme
que $\Sigma_g^{(n)}=\Sigma_f^{(n)}$, de sorte 
que $\Sigma_g^{(n)}$ est une partie paralinéaire de~$X$. 
Plus généralement:

\begin{lemm}
Soit $g\colon X\to\R^m$ une application paralinéaire sur~$X$.

\begin{enumerate}
\item Soit $q\colon\R^m \to\R^p$ une application affine injective.
Alors $\Sigma_{q\circ g}^{(n)}= \Sigma_g^{(n)}$.
\item
Il existe un G-recouvrement~$(W_j)$ de~$X$
et une famille~$(f_j)$, où $f_j\colon W_j\to \gm^n$ est un moment,
tels que $\Sigma_g^{(n)}\cap W_j\subset \Sigma_{f_j}$ pour tout~$j$.
\item 
L'ensemble $\Sigma_g^{(n)}$ est une partie paralinéaire de~$X$,
purement de dimension~$n$, sur laquelle $g$ est une immersion par morceaux.
\item
Si $X$ est compact, alors $g(X)^{(n)}=g(\Sigma_g^{(n)})$.
\end{enumerate}
\end{lemm}
\begin{proof}
\emph a)
En effet, comme $q$ est injective, $q\circ g$ et $g$ ont même
dimension tropicale en tout point.

Les énoncés~\emph b) et~\emph c) sont locaux sur~$X$,
ce qui permet de supposer que $X$ est paracompact.
Par définition d'une application paralinéaire,
il existe un G-recouvrement $(W_j)$ de~$X$ par des domaines compacts
et, pour tout~$j$, un moment 
$f_j\colon W_j\to \gm^{n}$,
un entier non nul~$d_j$
et un élement $a_j\in\R^m$
telle que $d_j g=f_{j,\trop}+a_j$. 
Puisque $X$ est paracompact, on peut supposer que ce G-recouvrement
est localement fini.

Pour tout~$j$, l'égalité $d_j g = f_{j,\trop}+a_j$ (et la non-nullité de~$d_j$)
entraîne que $\Sigma_{g|_{W_j}}^{(n)}=\Sigma_{f_j}^{(n)}$.
En particulier, il est contenu dans~$\Sigma_{f_j}$.
Soit $\Sigma$ la réunion des~$\Sigma_{f_j}$.
D'après la proposition~\ref{prop-unionsigmaf-paralineaire},
c'est une partie paralinéaire de~$X$;
de plus, il existe un G-recouvrement~$(V_i)$ de~$X$
et des moments~$h_i\colon V_i\to\gm^n$ tels que
$\Sigma\cap V_i\subset \Sigma_{h_i}$ pour tout~$i$.
Cela prouve~\emph b).

Démontrons~\emph c).
L'ensemble $\Sigma_g^{(n)}$ est la réunion
des $\Sigma_{g|_{W_j}} ^{(n)}$, donc est contenue dans~$\Sigma$.
Pour tout~$j$,
l'ensemble $\Sigma_{f_j}^{(n)}$ est une partie paralinéaire fermée
de~$\Sigma_{f_j}$;
par conséquent, $\Sigma_g^{(n)}$ est une réunion localement
finie de parties paralinéaires, donc est paralinéaire
en vertu du lemme~\ref{lemm.paralineaire-Glocal}.
Le reste de l'assertion
découle de la proposition~\ref{prop.image-squelette-pur},
de même que l'assertion~\emph d).
\end{proof}

\begin{defi}
Soit $X$ un espace $k$-analytique de niveau~$n$.
Soit $V$ un domaine de~$X$; 
on dit qu'une partie~$\Sigma$ de~$V$ est une \emph{partie squelettique
élémentaire} de~$V$ s'il existe un moment $f\colon V\to\gm^n$
tel que $\Sigma=\Sigma_f$.
\end{defi}

\begin{theo}
Soit $X$ un espace $k$-analytique de niveau~$n$.
Les parties squelettiques élémentaires de~$X$ et de ses domaines 
satisfont les axiomes des familles de parties squelettiques élémentaires
(\S\ref{defi.squel-elem}).
\end{theo}
\begin{proof}
Vérifions les axiomes~\emph a) à~\emph d) de la définition.
\begin{enumerate} 
\item
Si $V$ est un domaine analytique de~$X$ et $f\colon V\to\gm^n$ un
moment, alors $\Sigma_f=f^{-1}(S(\gm^n))$ est fermé dans~$V$
puisque $f$ est continu et $S(\gm^n)$ est fermé dans~$\gm^n$.

\item
Soit $V$ un domaine analytique de~$X$, soit $f\colon V\to\gm^n$
un moment et soit $W$ un domaine de~$V$. Alors, 
$\Sigma_{f}\cap W=\Sigma_{f|_ W}$, ce qui prouve que $\Sigma_f\cap W$
est une partie squelettique élémentaire de~$W$.

\item
C'est exactement le lemme~\ref{lemme-sigmaf-paralineaire}.

\item
Soit $(V,g)$ une carte G-tropicale de~$X$. 
Démontrons qu'il existe un G-recouvrement $(W_j)$ de~$V$
tel que pour tout~$j$, la partie
$\Sigma_f^{(n)}\cap W_j$ de~$W_j$ soit contenue
dans une partie squelettique élémentaire de~$W_j$.

L'assertion est locale sur~$X$;
cela permet de supposer que $X$ est paracompact.
On choisit maintenant un G-recouvrement $(W_j)$ de~$V$
par des domaines compacts, une famille~$(d_j)$
d'entiers non nuls, une famille $(a_j)$ dans~$\R^m$ et une
famille de moments $(f_j)$ tels que
$d_j g_\trop|_{W_j}=f_{j,\trop}+a_j$ pour tout~$j$.
Puisque $X$ est paracompact, on se ramène au cas où le G-recouvrement~$(W_j)$
est localement fini.

La preuve du lemme entraîne que $\Sigma_g^{(n)}$
est contenu dans la réunion des~$\Sigma_{f_j}$.
D'après la proposition~\ref{prop-unionsigmaf-paralineaire},
cette réunion est G-localement contenue dans
une partie squelettique élémentaire.
\qedhere
\end{enumerate}
\end{proof}

\subsection{}\label{ss.an-G-tropical-vertebre}
La théorie générale des parties squelettiques fournit
maintenant une notion de partie squelettique d'un espace analytique
associée à ces parties squelettiques élémentaires.

Soit $X$ un espace $k$-analytique de niveau~$n$.
Une partie~$\Sigma$ de~$X$ est squelettique s'il existe
une famille~$(V_i)$ de domaines analytiques de~$X$
telle que les $\Sigma\cap V_i$ forment un G-recouvrement de~$\Sigma$
et telle que pour tout~$i$, $\Sigma\cap V_i$ soit une partie
paralinéaire fermée d'une partie squelettique élémentaire de~$V_i$.

Notons que toute partie squelettique de~$X$
est constituée de points~$x$ tels que $d_k(x)=n$;
en un tel point~$x$, on a $\dim_x(X)=n$.
En particulier, si $\dim(X)<n$, la partie vide
est la seule partie squelettique de~$X$.

Ainsi, tout espace analytique de dimension~$\leq n$
hérite d'une structure d'espace G-tropical vertébré de niveau~$n$.
\index{espace G-tropical vertébré associé à un espace analytique}

\begin{prop}\label{prop.sigmanf-fini}
Soit $X$ et~$Y$ des espaces $k$-analytiques de niveau~$n$
et soit $p\colon Y \to X$ un morphisme.

\begin{enumerate}
\item
Soit $\Sigma$ une partie squelettique élémentaire d'un domaine~$V$ de~$X$.
Alors $p^{-1}(\Sigma)$ est une partie squelettique élémentaire de~$p^{-1}(V)$.
\item
Soit $\Sigma$ une partie squelettique de~$X$.
Alors $p^{-1}(\Sigma)$ est une partie squelettique de~$Y$,
et l'application~$p_\Sigma\colon p^{-1}(\Sigma)\to\Sigma$
est une immersion par morceaux.
\item
Supposons de plus que $p^{-1}(\Sigma)\subset \Int(Y/X)$.
L'application~$p_\Sigma$ est alors ouverte ; 
en particulier, elle préserve la dimension locale.
\end{enumerate}
\end{prop}
\begin{proof}
\emph a)
Par hypothèse, il existe un moment $f\colon V\to \gm^n$
tel que $\Sigma=f^{-1}(S(\gm^n))$.
Alors $p^{-1}(\Sigma)=(f\circ p)^{-1}(S(\gm^n))$
donc est une partie squelettique élémentaire de~$p^{-1}(V)$.

Les assertions restantes sont G-locales sur~$X$.
On peut donc supposer qu'il existe un moment $f\colon X\to \gm^n$
tel que $\Sigma\subset\Sigma_f$. 

\emph b)
Alors, $\Sigma_{f\circ p}=p^{-1}(\Sigma_f)$, par définition de ces ensembles.
Dans le diagramme commutatif
\[\begin{tikzcd}
{\Sigma_{f\circ p}}\ar[rr, "{p_{\Sigma_{f}}}"] \ar[rd,"{(f\circ p)_\trop}"']
	&&{\Sigma_f}\ar[ld, "{f_\trop}"] \\
&{T_\trop}&
\end{tikzcd}\]
les deux flèches obliques sont des immersions par morceaux. 
De plus, $p^{-1}(\Sigma)$ est l'image réciproque
de~$\Sigma$ par l'immersion par morceaux de~$\Sigma_{f\circ p}$
dans~$\Sigma_f$ déduite de~$p$.
Ainsi, $p^{-1}(\Sigma)$ est une partie paralinéaire
de~$\Sigma_{f\circ p}$ et
$p_\Sigma$ est une immersion par morceaux.

\emph c)
L'assertion est locale sur $p^{-1}(\Sigma)$ et a fortiori sur~$\Sigma$.
Soit $y\in p^{-1}(\Sigma)$. 
Alors $p$ est de dimension relative nulle et sans bord en~$y$, donc fini en~$y$.
Quitte à restreindre~$X$ et~$Y$, on peut supposer que $p$ est fini.
Considérons le diagramme commutatif d'espaces $k$-analytiques:
\[ \begin{tikzcd}
   Y_\red \ar{r}{p_\red} \ar{d} & X_{\red} \ar[dr] \ar{d} \\
   Y \ar{r}{p} & X \ar{r}{f} & T .\end{tikzcd} \]
L'immersion de~$X_\red$ dans~$X$ induit un homéomorphisme
de $\Sigma_{f_\red}$ sur~$\Sigma_f$;
de même, l'immersion de~$Y_\red$ dans~$Y$ induit un homéomorphisme
de~$\Sigma_{(f\circ p)_\red}$ sur~$\Sigma_{f\circ p}$.
Ainsi, il suffit de traiter le cas où $X$ et~$Y$ sont réduits.

L'ensemble~$Y'$ des points~$y$ de~$Y$ en lesquels $p$ est plat 
est un ouvert de Zariski de~$Y$.  Il n'est pas vide: en effet, si $V$
est un domaine affinoïde non vide de~$X$ et~$x$ un point de~$X$
en lequel  l'anneau local~$\mathscr O_{V,x}$ est artinien,
cet anneau local est un corps (car $X$ est réduit)
et le morphisme fini~$Y\times_X V\to V$ est plat au-dessus de~$x$.
Le fermé de Zariski $Y\setminus Y'$ est alors de dimension~$<n$,
donc son image $p(Y\setminus Y')$ est un fermé de Zariski  de~$X$
de dimension~$<n$.
Notons~$X'$ l'ouvert complémentaire;
il contient $\Sigma_f$, car $d_k(x)=n$ pour tout $x\in \Sigma_f$. 
Par construction, le morphisme $Y\times_XX'\to X'$ est fini et plat. 
Il est en particulier ouvert.
Par conséquent, l'application~$p_\Sigma$ est ouverte.
\end{proof}

\begin{prop}\label{prop-squelettes-nilpotents}
Soit $X$ un espace $k$-analytique de niveau~$n$.  
Pour qu'une partie~$\Sigma$ de~$X$ soit squelettique,
il faut et il suffit qu'elle soit squelettique
en tant que partie de~$X_\red$.
\end{prop}
\begin{proof}
Soit $\Sigma$ une partie squelettique de~$X$;
d'après la proposition précédente, c'est également
une partie squelettique de~$X_\red$.

Inversement, soit $\Sigma$ une partie squelettique de~$X_\red$
et démontrons que c'est une partie squelettique de~$X$.
L'assertion est G-locale sur~$X$; cela permet de supposer
que $X$ est affinoïde.
On peut alors choisir un recouvrement affinoïde fini~$(V_i)$
de~$X_\red$ tel que, pour tout~$i$,
$\Sigma\cap V_i$ soit contenue dans une partie squelettique
élémentaire~$\Sigma_{f_i}$ de~$V_i$.
D'après le théorème de Gerritzen--Grauert, on peut supposer en outre
que chaque~$V_i$ est un domaine rationnel de~$X_\red$;
c'est alors la trace d'un domaine rationnel~$W_i$ de~$X_\red$,
de sorte que $V_i=(W_i)_\red$.
Pour tout~$i$, il existe alors un moment~$g_i$ sur~$W_i$
induisant~$f_i$ par réduction. On a alors $\Sigma_{f_i}=\Sigma_{g_i}$,
ce qui prouve que $\Sigma\cap W_i$ est contenu
dans une partie squelettique élémentaire de~$W_i$.
Comme les~$W_i$ forment un G-recouvrement affinoïde fini de~$X$,
cela démontre que $\Sigma$ est une partie squelettique de~$X$.
\end{proof}

\begin{rema}
Soit $X$ un espace $k$-analytique de niveau~$n$,
soit $m$ un entier tel que $m\leq n$,
soit $Y$ un fermé de Zariski de~$X$, considéré 
comme espace $k$-analytique de niveau~$m$.
La notion de partie squelettique de~$Y$ 
ne dépend pas de la structure analytique précise choisie sur~$Y$;
elle coïncidera avec la notion fournie par sa structure réduite.
Si $\dim(Y)<n$, sa seule partie squelettique est l'ensemble vide.
\end{rema}

\begin{lemm}\label{lemm.polyedre-irreductible}
Soit $X$ un espace $k$-analytique de niveau~$n$.
Soit $(X_i)$ la famille de ses composantes irréductibles  de~$X$;
pour tout~$i$, considérons~$X_i$ comme un espace
$k$-analytique de niveau~$n$.
Soit $\Sigma$ une partie de~$X$. Les assertions suivantes
sont équivalentes:
\begin{enumerate}\def\theenumi{\roman{enumi}}\def\labelenumi{(\theenumi)}
\item L'ensemble~$\Sigma$ est une partie squelettique de~$X$;
\item Pour tout~$i$, l'ensemble~$\Sigma\cap X_i$ 
est une partie squelettique de~$X_i$;
\item Le sous-espace~$\Sigma$ est somme disjointe des ensembles
$\Sigma\cap X_i$, chacune d'elles étant squelettique dans~$X_i$.
\end{enumerate}
\end{lemm}
\begin{proof}
Soit $U$ l'ensemble des points de~$X$ 
qui n'appartiennent qu'à une seule composante irréductible de~$X$.
Observons que $U$ est un ouvert de Zariski de~$X$. 

Supposons que $\Sigma$ soit une partie  squelettique de~$X$.
Comme un point d'un polyèdre caractéristique n'est contenu dans
aucun fermé de Zariski de dimension~$<n$, on a $\Sigma\subset U$.
Par suite, $\Sigma$ est une partie squelettique de~$U$;
comme  $U\cap X_i$ est un ouvert de~$U$, $\Sigma\cap X_i$ est donc une partie
squelettique de~$U\cap X_i$;
comme $U\cap X_i$ est un ouvert de~$X_i$, $\Sigma\cap X_i$ est 
une partie squelettique de~$X_i$. Cela démontre l'implication
(i)$\Rightarrow$(ii).

Supposons que pour tout~$i$, $\Sigma\cap X_i$ soit une partie
squelettique de~$X_i$. 
Pour tout~$i$, $\Sigma\cap X_i$ est alors contenu dans~$U$,
d'où~(iii).

Sous les hypothèses de~(iii),
$\Sigma\cap X_i$ est une partie squelettique de~$U\cap X_i$,
donc de~$U$ puisque $U\cap X_i$ est ouvert dans~$U$,
et donc de~$X$ puisque $U$ est ouvert dans~$X$.
Comme $\Sigma$ est la réunion disjointe des parties
ouvertes et fermées $\Sigma\cap X_i$,  
c'est une partie squelettiques de~$X$.
\end{proof}

La proposition suivante indique que dans le cas d'un
bon espace analytique de dimension~$n$,
les parties squelettiques 
coïncident avec les parties qui sont, localement,
parties paralinéaires d'un polyèdre caractéristique de moment
de taille~$n$.
 
\begin{prop}\label{prop.parties-polyedrales-moments}
Soit~$X$ un espace analytique de niveau~$n$, 
soit~$\Sigma$ une partie squelettique de~$X$ 
et soit $x$ un point de~$\Sigma$ en lequel le germe $(X,x)$ est presque bon.
Il existe un voisinage ouvert~$U$ de~$x$ et
un moment $f\colon U\to\gm^n$ tel que $\Sigma\cap U\subset\Sigma_f$.
\end{prop}
\begin{proof}
Par définition, il existe une famille finie $(V_i)$ de domaines analytiques
compacts contenant~$x$, dont la réunion~$V$ est un voisinage de~$x$,
et, pour tout~$i$, un moment $f_i\colon V_i\to \gm^{n}$,
tel que $\Sigma\cap V_i\subset \Sigma_{f_i}$.
Pour tout $i$ et tout entier~$j$ tel que $1\leq j\leq n$, 
choisissons un voisinage analytique compact~$W_i$ de~$x$ dans~$V$
et une fonction analytique inversible~$g_{i,j}$ sur~$W_i$
telle que  $\abs{g_{i,j}-f_{i,j}}<\abs{f_{i,j}}$ sur $W_i\cap V_i$;
c'est possible car $X$ est presque bon en~$x$ (lemme~\ref{lemm.dense-moment}).
D'après le paragraphe~\ref{imrecisquel-explicite}, 2),
cela entraîne l'égalité $\Sigma_{g_i}\cap V_i=\Sigma_{f_i}\cap W_i$,
où $g_i$ est le moment sur~$W_i$ défini par la concaténation des~$g_{i,j}$.
Soit~$W$ l'intersection des~$W_i$ et soit~$g\colon W\to\gm^m$ 
le moment sur~$W$
obtenu par concaténation des~$g_i$. L'ensemble~$W$
est un voisinage analytique compact de~$x$. Comme
$W\subset V = \bigcup_i V_i$, on a 
\[ \Sigma\cap W\subset\bigcup \Sigma_{f_i}\cap W 
\subset\bigcup \Sigma_{f_i}\cap W_i \subset \bigcup \Sigma_{g_i}\cap V_i
\subset\Sigma_g .\]
On conclut à l'aide de la proposition~\ref{prop.moment-bonne-proj}.
%
\end{proof}

\begin{theo}\label{theo.chgt-base-polyedre}
Soit~$X$ un espace $k$-analytique de niveau~$n$.
Soit~$L$ une extension complète de~$k$ et notons $p\colon X_L\to X$
la projection canonique.
Soit~$\Sigma$ une partie squelettique de~$X$;
notons~$\Sigma_L$ la réunion des polyèdres caractéristiques~$\Sigma_{f_L}$,
où $f\colon V\to T$ parcourt l'ensemble des moments définis sur un domaine
analytique~$V$ de~$X$ tels que $\Sigma_f\subset\Sigma$.

L'ensemble~$\Sigma_L$ est 
une partie squelettique de~$X_L$
et le morphisme~$p$
induit une immersion par morceaux surjective de~$\Sigma_L$ sur~$\Sigma$
qui préserve la dimension locale.

De plus, pour tout moment~$f\colon V\to T$ tel que $\Sigma_f\subset\Sigma$,
on a $\Sigma_L\times_\Sigma \Sigma_f= \Sigma_{f_L}$.
\end{theo}
\begin{proof}
La première assertion est G-locale sur~$X$; comme elle est vérifiée
pour l'application~$\Sigma_{f_L}\to \Sigma_f$, il suffit
de démontrer la seconde assertion.

Soit $f\colon V\to T$ un moment
défini sur un domaine analytique~$V$ de~$X$.
Démontrons l'égalité $P_L\times_P \Sigma_f=\Sigma_{f_L}$.
Cela revient à démontrer que si $g\colon W\to T'$ est un moment 
et $x\in \Sigma_f\cap \Sigma_g$, alors les fibres de~$\Sigma_{f_L}$
et~$\Sigma_{g_L}$ au-dessus de~$x$ coïncident.
Quitte à remplacer~$X$ par $V\cap W$, on suppose que les moments~$f$
et~$g$ sont définis sur~$X$.

Soit $x\in \Sigma_f\cap \Sigma_g$.
Soit $C$ la complétion d'une clôture algébrique de~$k$
et soit $M$ une extension complète composée de~$L$ et~$C$.
Considérons le diagramme
\[ \begin{tikzcd}[column sep=small]
   & X_M \ar{rd} \ar{ld}[swap]{q'} \\
X_L \ar{rd}[swap]{p} && X_C \ar{ld}{q} \ar[bend right=30]{ul}[swap]{\sigma} \\
& X \end{tikzcd} \]
dans lequel $\sigma$ est la section de Poineau 
de l'application $X_M\to X_C$ dont l'existence résulte
de ce que $C$ est algébriquement clos.
Comme l'extension $C/k$ contient une sous-extension algébrique
dense, on a $\Sigma_{f_C}=q^{-1}(\Sigma_f)$.
Comme $C$ est algébriquement clos, on a $\Sigma_{f_M}=\sigma(\Sigma_{f_C})$.
Enfin, on a $\Sigma_{f_L}=q'(\Sigma_{f_M})$. Par suite,
\[ \Sigma_{f_L} = q'( \sigma (q^{-1}(\Sigma_f))), \]
donc la fibre de~$\Sigma_{f_L}$ au-dessus de~$x$ est donnée par la formule
\[ \Sigma_{f_L,x} = q'( \sigma (q^{-1}(x))) \]
dans laquelle $\Sigma_f$ n'intervient plus ;
cette fibre coïncide donc avec la fibre de~$\Sigma_{g_L}$ au-dessus de~$x$.
Cela conclut la démonstration du théorème.
\end{proof}

Le cas des polyèdres caractéristiques et la définition de~$\Sigma_L$
entraînent
de plus la caractérisation suivante:
\begin{prop}\label{prop.chgt-base-polyedre}
Conservons les notations du théorème~\ref{theo.chgt-base-polyedre}.
\begin{enumerate}
\item
Si l'extension~$L/k$ possède une sous-extension algébrique dense,
alors $\Sigma_L=p^{-1}(\Sigma)$;
\item
Si l'extension~$L/k$ est universellement multiplicative,
ou si le corps le corps~$k$ est algébriquement clos,
on a $\Sigma_L=\sigma_L(\Sigma)$.
\end{enumerate}
\end{prop}
\begin{proof}
Compte tenu de la seconde assertion du théorème,
l'assertion~\emph a) découle de la proposition~\ref{prop.polyedre-chgt-base-algebrique},
et l'assertion~\emph b) du corollaire~\ref{coro.Sigma-chgt-base}.
\end{proof}

\begin{lemm}\label{lemme-fibreetar-pl}
Soit $X$ un espace $k$-analytique de niveau $n$, soit $\Sigma$ une partie squelettique de $X$
et soit $d$ un entier $\leq n$. Soit $f\colon X\to \A^d$ un morphisme et soit $r$ un polyrayon de longueur $d$. 
La fibre $f^{-1}(\eta_r)$ est de dimension $\leq n-d$ et l'intersection $\Sigma\cap f^{-1}(\eta_r)$ est une partie paralinéaire
de $\Sigma$ de dimension $\leq n-d$ ; c'est aussi une partie squelettique de $f^{-1}(\eta_r)$ vu comme 
espace de niveau $n-d$. 
\end{lemm}

\begin{proof}
Comme $X$ est de niveau $n$, il est de dimension $\leq n$. 
Pour tout $x\in f^{-1}(\eta_r)$ on a donc $n\geq d_k(x)=d_{\hr{\eta_r}}(x)+d$, 
si bien que $d_{\hr{\eta_r}}(x) \leq n-d$. 
Par conséquent $\dim f^{-1}(\eta_r)\leq d$. 
Dans ce qui suit, nous considérons cette fibre comme un espace 
$\hr{\eta_r}$-analytique de niveau $n-d$. 

Pour démontrer le lemme, on peut raisonner G-localement sur~$\Sigma$, 
et donc supposer que $\Sigma$ est compact
et qu'il existe un domaine compact~$V$ de~$X$
et un moment $g\colon V\to \gm^n$ tel que $\Sigma \subset  \Sigma_g$.
Il existe alors un domaine compact~$W$ de~$V$ tel que $W\cap \Sigma_g=\Sigma$
(lemme~\ref{lemme-proprietes-paralineairesX}).
Quitte à remplacer~$X$ par~$W$, cela permet de supposer
que $X$ est compact, que $g$ est défini sur~$X$ et que $\Sigma=\Sigma_g$.

Comme l'énoncé est trivial au voisinage d'un point de $\Sigma$ 
qui n'est pas dans $f^{-1}(\eta_r)$,
on peut aussi supposer $f$ inversible sur $X$. 

Soit $x\in \Sigma_g\cap f^{-1}(\eta_r)$. On a alors 
 $\abs{f(x)}=r$ et $\widetilde{f(x)}$ est une famille de $d$ éléments algébriquement
indépendants sur $\widetilde k$. Comme $x$ appartient à $\Sigma_g$, 
 le corps gradué
 $\hrt x$ est de degré de transcendance $n$ sur $\widetilde k$ et la famille $\widetilde{g(x)}$
 est une base de transcendance
 de $\hrt x$ sur $\widetilde k$. Mais alors le sous-corps
 gradué de $\hrt x$ engendré par $\widetilde k$ et les
 familles $\widetilde{f(x)}$ et $\widetilde{g(x)}$ 
 est a fortiori de degré de transcendance~$n$ sur~$\widetilde k$, et donc de degré de transcendance $n-d$
 sur $\widetilde k(\widetilde{f(x)})$. 

Il existe dès lors
 un sous-ensemble~$J$ de $\{1,\ldots, n\}$ de cardinal~$n-d$
 tel que la famille $\widetilde{g_J(x)}$ forme
 une base de transcendance de~$\hrt x$ sur~$\widetilde  k(\widetilde{f(x)})$, où $g_J$ désigne la sous-famille de~$g$ 
constituée des composantes d'indice appartenant à $J$. 

 Le point $x$ appartient donc à $\Sigma_{g_J,f}\cap f^{-1}(\eta_r)$, qui est aussi le sous-espace
 paralinéaire~$\Tau_J$ de~$\Sigma_{g_J, f}$ défini par l'égalité $\abs f=r$, puisque $f(\Sigma_{g_J, f})$ est à valeurs
 dans le squelette standard de $\gm^d$.  

Si $y$ est un point de $f^{-1}(\eta_r)$, la famille $\widetilde{f(y)}$ 
est algébriquement indépendante sur $\widetilde k$;
par suite, $\widetilde{g_J(y)}$ est algébriquement indépendante 
sur~$\widetilde k(\widetilde {f(y)}) =\hrt{\eta_r}$ 
si et seulement si la famille concaténée 
$(\widetilde{g_J(y)}, \widetilde{f(y)})$ est algébriquement
indépendante sur~$\widetilde k$. 
Autrement dit, $\Tau_J=\Sigma_{g_J|_{f^{-1}(\eta_r)}}$.

Remarquons que $\Tau_J$ apparaît comme une partie squelettique
de~$X$ aussi bien que de~$f^{-1}(\eta_r)$ ; les structures
paralinéaires correspondantes coïncident 
car il est G-recouvert par des $g$-cellules.

D'après ce qui précède, 
$\Sigma_g\cap f^{-1}(\eta_r)= \bigcup_J \Sigma_g\cap \Tau_J$;
c'est donc une partie squelettique de~$X$
aussi bien que de $f^{-1}(\eta_r)$.
Comme $\dim(f^{-1}(\eta_r))\leq n-d$,
la dimension de~$\Sigma_g\cap f^{-1}(\eta_r)$ 
est au plus égale à~$n-d$.
\end{proof}

\begin{coro}\label{coro-bor-petit}
Soit $X$ un espace $k$-analytique de niveau~$n$ 
et soit~$\Sigma$ une partie squelettique de~$X$. 
L'intersection $\partial(X)\cap \Sigma$ est $n$-négligeable ;
si $X$ est affinoïde, c'est une partie paralinéaire
de~$\Sigma$ de dimension $\leq n-1$. 
\end{coro}
\begin{proof}
Pour démontrer que $\partial(X)\cap\Sigma$
est $n$-négligeable, il suffit 
de vérifier que $V\cap \Sigma\cap \partial X$ est $n$-négligeable 
pour tout domaine affinoïde~$V$ de~$X$. 
Comme $\partial V\supset \partial X\cap V$, il suffit
de traiter le cas où $X$ est affinoïde. 

D'après le théorème~3.2 de~\cite{ducros2012b},
il existe une famille finie~$(f_i)$ de fonctions analytiques sur~$X$
et une famille~$(r_i)$ de nombres réels strictement positifs
telles que $\partial(X)=
\bigcup_i f_i^{-1}(\eta_{r_i})$.
Le corollaire découle alors immédiatement du lemme ci-dessus. 
\end{proof}

\section{Les espaces de Berkovich sont riches}

\subsection{}
Nous nous proposons de montrer que tout espace analytique 
topologiquement séparé et purement de dimension~$n$
est riche en tant qu'espace G-tropical de niveau~$n$.
La preuve va reposer sur un certain nombre de résultats intermédiaires 
dans lesquels
les valuations de Gauss jouent un rôle central. 
Pour manipuler commodément ces dernières, 
nous utiliserons la notation multiplicative ; 
la tropicalisation de $\gm^n$ sera en particulier 
dans ce qui suit identifiée à $(\R_+^\times)^n$. 

Nous aurons besoin d'une variante graduée des valuations de Gauss. 
Plus précisément, soit $K$ un corps gradué et
soit $K(\tau_1,\ldots, \tau_n)$ une extension transcendante pure de~$K$ 
(les $\tau_i$ sont donc homogènes, et
algébriquement indépendants sur $K$). 
Soit $H$ un groupe ordonné (noté multiplicativement) et soit $h\in H^n$. 
La valuation de Gauss de paramètre $h$ sur $K(\tau_1,\ldots, \tau_n)$, relative à la famille $(\tau_i)$, est 
la valuation à valeurs dans~$H$ donnée par la formule $\sum a_I T^I/\sum b_I T^I\mapsto \max_I{v_0(a_I)}/\max_I v_0(b_I)$, 
pour tout couple $(\sum a_I T^I, \sum b_I T^I)$ d'éléments homogènes de $K[\tau_1,\ldots,\tau_n]$ avec les $b_I$ non tous nuls, 
où $v_0$ est la valuation triviale sur~$K$ (qui envoie $0$ sur $0$ et tout élément de $K^\times$ sur $1$). 

Notons $\mathsf D$  la catégorie des groupes abéliens divisibles ordonnés non triviaux et $\Gamma$ le foncteur d'oubli de $\mathsf D$
vers la catégorie des ensembles (autrement dit, $\mathsf D=\mathsf D_{\{1\}}$ et $\Gamma=\Gamma_{\{1\}}$ avec les notations de 
\ref{ss.vocabulaire-foncteur-polyedral}). 
Un \textit{cône rationnel} de $\Gamma^n$
sera un sous-foncteur polyédral de $\Gamma^n$ au sens de \ref{ss.vocabulaire-foncteur-polyedral}, 
c'est-à-dire un sous-foncteur de $\Gamma^n$ décrit par une condition de la forme
$\bigvee_i \bigwedge_j \phi_{ij}\leq 1$ où chaque $\phi_{ij}$ est un monôme de la forme $x_1^{e_1}\ldots x_n^{e_n}$, les exposants $e_i$ étant
rationnels (le groupe ordonné de base est ici égal à $\{1\}$ et il n'y a donc pas de termes constants ; un cône rationnel non vide de $\Gamma^n$
contient automatiquement l'origine $(1,\ldots, 1)$).

\begin{lemm}\label{lem-squelette-tilde}
Soit $n$ un entier et soit $r=(r_1,\ldots, r_n)$ un polyrayon. 
Posons
$\tau_i=\widetilde{T_i(\eta_r)}$ pour tout $i$. 
Soit $(U,\eta_r)$ un germe de domaine analytique de $\gm^n $
et soit $C\subset\mathbf G^n$ l'unique cône rationnel 
tel que $\trop(U\cap S(\gm^n))=r\cdot C(\R_+^\times)$ au voisinage de~$r$.
Soit $H$ un objet de $\mathsf D$, soit $h\in H^n$
et soit $\eta_h$ la valuation de Gauss sur l'extension $\hrt{\eta_r}$
de $\widetilde k$, de paramètre~$h$
et relative à la famille $(\tau_i)$. 
La valuation $\eta_h$ appartient à $\widetilde {(U,\eta_r)}$ 
si et seulement si $h\in C(H)$. 
\end{lemm}
\begin{proof}
Les formations de la réduction $\widetilde {(U,\eta_r)}$ et du cône $C$ 
commutent aux unions finies de germes de domaines analytiques; 
on peut supposer que le germe~$(U,\eta_r)$ est bon, 
c'est-à-dire que $\widetilde{(U,\eta_r)}$ est de la forme
$\mathbf P_{K/\widetilde k}\{\widetilde{f_1(\eta_r)},\ldots, \widetilde{f_m}(\eta_r)\}$, où les 
$f_i$ sont des fonctions analytiques inversibles définies au voisinage de $\eta_r$. Par
densité de $k(T_1,\ldots, T_n)$ dans $\hr {\eta_r}$, on peut supposer que chaque
$f_i$ est de la forme $g_i/h_i$, où $g_i$ et $h_i$ sont des polynômes 
et où $h_i\neq 0$. 
Comme seules importent les réductions graduées des
$f_i(\eta_r)$, on peut supposer qu'il existe pour tout $i$ deux réels strictement positifs
$\alpha_i$ et $\beta_i$ 
et des écritures $f_i=\sum_{I\in E_i}  a_{I,i}T^I$ et $g_i=\sum_{I\in F_i} b_{I,i}T^I$
où $E_i$ et $F_i$ sont finis et non vides
et où $\abs{a_{I,i}}r^I=\alpha_i$
pour tout $I\in E_i$ et $\abs{b_{I,i}}r^I=\beta_i$ 
pour tout $I\in F_i$ ; 
notons qu'on a alors $\abs{f_i(\eta_r)} =\alpha_i/\beta_i$ . 

Au voisinage de~$\eta_r$, le domaine $U$ est défini par la conjonction d'inégalités 
$\abs {f_i}\leq \alpha_i/\beta_i$. 
Soit $s$ un $n$-uplet de nombres réels strictement positifs, 
et posons $t=sr^{-1}$. Si $s$ est suffisamment proche de~$r$ le point
$\eta_s$ appartient à $U$ si et seulement si
 $\abs{f_i(\eta_s)}
\leq \alpha_i/\beta_i$ pour tout $i$, ce qui se récrit
$\max_{I\in E_i} t^I\leq \max_{I\in F_i}t^I$. Le cône rationnel $C$ est donc
décrit par la conjonction des inégalités
$\max_{I\in E_i} t^I\leq \max_{I\in F_i}t^I$,
pour $i$ variant entre $1$ et $n$. 

On déduit par ailleurs de la description 
de $(U,\eta_r)$ 
que 
\[\widetilde {(U,\eta_r)}=\zr{\hrt{\eta_r}}{\widetilde k}\left\{\frac{\sum_{I\in E_i}\widetilde{a_{I,i}}\tau ^I}
{\sum_{I\in F_i}\widetilde{b_{I,i}}\tau ^I}\right\}_{1\leq i\leq n}.
\]
Il en résulte que $\eta_h$ appartient à $\widetilde {(U,\eta_r)}$
si et seulement si 
$\max_{I\in E_i} h^I\leq \max_{I\in F_i}h^I$, c'est-à-dire si et seulement si 
$h\in C(H)$.
\end{proof}

\subsection{}
Soit $K$ un corps gradué et soit $L$ une extension graduée de $K$. 
Soit
$f=(f_1,\ldots,f_n)$ un élément de $(L^\times)^n$. 
Si $v\in \zr LK$
nous écrirons $f_\trop(v)$ au lieu de $(v(f_1),\ldots, v(f_n))$
et si $D$ est un sous-foncteur définissable de 
$\Gamma^n$, nous écrirons «$f_\trop(v)\in  D$» pour
signifier que $f_\trop(v)$ appartient à $D(H)$ où $H$ est un objet
de $\mathsf D$ dans lequel $v$ prend ses valeurs (la validité cet énoncé est 
indépendante du choix de $H$). Et nous noterons $f_\trop^{-1}(D)$
l'ensemble des $v\in \zr LK$ telles que $f_\trop(v)\in D$ ; c'est une partie
constructible de $\zr LK$. 
Par exemple
si $D$ est décrit par les inégalités $x_i\leq 1$ pour $i=1,\ldots n$ alors
$f_\trop^{-1}(D)$ est égal à $\zr LK\{f_1,\ldots, f_n\}$. En général si $D$
est décrit par une conjonction d'inégalités larges 
$f_\trop^{-1}(D)$ est un ouvert affine de $\zr LK$, et si $D$ est 
un cône rationnel c'est un ouvert quasi-compact de $\zr LK$.

Soit $\mathscr U$ un ouvert quasi-compact de $\zr LK$. 
D'après le théorème \ref{theo.foncteurs-polyedraux}, il existe un cône
rationnel $C$ de $\mathbf G^m$ tel que pour tout $H\in \mathsf D$ 
et tout $m$-uplet $h=(h_1,\ldots, h_m)$ 
d'éléments de $H$
les assertions suivantes soient équivalentes : 
\begin{itemize}
\item[(i)] il existe $H'\in \mathsf D_H$
et une valuation $w\in \mathscr U$ à valeurs dans $H'$ telle que
$f_{i,\trop}(w_i)=h_i$ pour tout $i$ ; 
\item[(ii)] $h\in C(H)$. 
\end{itemize}
Le cône $C$ sera noté $f_\trop(\mathscr U)$. 

\begin{lemm}\label{lem-puredim-squelette}
Soit $K\hookrightarrow K(\tau_1,\ldots, \tau_n)$ 
une extension transcendante pure de corps
gradués (chaque $\tau_i$ étant homogène) 
et soit $\mathscr U$ un ouvert quasi-compact et non vide
de $\mathbf P_{K(\tau)/K}$. Supposons que le cône
rationnel $\tau_\trop(\mathscr U)$ est purement de dimension~$n$. 
Alors pour tout $H\in \mathsf D$ et tout 
$h\in \tau_\trop(\mathscr U)(H)$ 
la valuation $\eta_h$ de $K(\tau)$ appartient à~$\mathscr U$. 
\end{lemm}

\begin{proof}
Si $f=\sum a_I \tau_I/\sum b_I \tau_I$
est un élément de $K(\tau)^\times$, 
alors pour tout $H\in \mathsf D$ et tout $h\in H^n$ on a 
$f_\trop(\eta_h)=\max_{I, a_I\neq 0}h^I/\max_{I, b_I\neq 0}h^I$, si bien que
$\eta_h$
appartient à $\zr {K(\tau)}K\{f\}$ si et seulement si 
$\max_{I,a_I\neq 0}h^I\leq \max_{I,b_I\leq 0} h^I$.

Il s'ensuit, l'ouvert $\mathscr U$ étant une combinaison booléenne positive
de parties de la forme $\zr {K(\tau)}K\{f\}$ avec $f\in K(\tau)^\times$, qu'il existe un cône
rationnel $D$ de $\Gamma^n$ tel que pour tout $H\in \mathsf D$ et tout $h\in H^n$, 
la valuation $\eta_h$ appartient à $\mathscr U$ si et seulement si 
$h$ appartient à $D(H)$. Si c'est le cas on déduit du fait
que $h_i=\tau_{i,\trop}(\eta_h)$ pour tout $i$ que
$h\in \tau_\trop(\mathscr U)(H)$. 
Par conséquent $D\subset \mathscr \tau_\trop(\mathscr U)$, et il reste à établir la réciproque. 
Autrement dit, nous voulons montrer que 
le sous-foncteur $E:=\tau_\trop(\mathscr U)\setminus D$ de $\Gamma^n$
est vide. 

Soit $H\in \mathsf D$ et soit $h=(h_1,\ldots, h_m)\in E(H)$. 
Par définition, il existe une valuation~$u$ située sur~$\mathscr U$
à valeurs dans un groupe abélien ordonné contenant $H$ 
telle que $\tau_\trop(u)=h$ et telle que $u\neq \eta_h$. 
Il en résulte l'existence d'un polynôme 
$P=\sum a_I \tau^I\in K(\tau)^\times$
tel que $P_\trop(u)<\max_{I, a_I\neq 0}h^I$. 
Ceci entraîne l'existence de 
deux multi-indices $I$ et $J$ distincts tels que
$a_I\neq0$, $a_J\neq0$ et $h^I=h^J$. 

Soit $E_{I,J}$ le cône rationnel de $\Gamma^n$ défini par l'égalité
$x^I=x^J$. 
Par ce qui précède, pour tout $H\in \mathsf D$ et tout $h\in E(H)$, il existe
deux indices $I$ et $J$ distincts tels que $h\in E_{I,J}(H)$. 
Par le théorème de compacité en théorie des modèles, 
$E$ est contenue dans une réunion finie $\bigcup_\ell E_{I\ell, J\ell}$,
laquelle est 
un cône rationnel de dimension $\leq n-1$
de $\Gamma^n$. 
En tant que sous-foncteur 
de $\tau_\trop(\mathscr U)$, le foncteur $E$ est le 
complémentaire de $D$ et est donc défini par une combinaison
booléenne positive d'inégalités strictes ; il est en conséquence
purement de dimension~$n$ puisque c'est le cas
de $\tau_\trop(\mathscr U)$ ; étant par ailleurs
contenu par ce qui précède
dans un cône rationnel de dimension $n-1$, le foncteur~$E$ est vide. 
\end{proof}

\begin{prop}\label{prop-tropicalisation-squelette}
Soit $X$ un espace $k$-analytique compact de niveau~$n$, 
soit $x$ un point de~$X$ tel que $d_k(x)=n$ 
et soit $f\colon X\to \gm^m$ un moment sur~$X$.
Il existe une partie squelettique~$\Sigma$ de~$X$ contenant~$x$
telle que les germes de cônes $f_\trop(X,x)$ et $f_\trop(\Sigma,x)$ coïncident. 
\end{prop}
\begin{proof}
Écrivons $f=(f_1,\ldots, f_m)$
et posons $\phi_i=\widetilde{f_i(x)}$ pour tout $i$. 
L'assertion à montrer est G-locale sur le germe $(X,x)$, 
ce qui permet de supposer que $X$ est affinoïde. 
Quitte à composer ensuite avec une projection convenable, on peut
également étendre le moment donné~$f$.
On peut ainsi supposer que la famille $(\phi_i)$ 
contient une base de transcendance de~$\hrt x$ sur~$\widetilde k$, 
c'est-à-dire que $x\in \Sigma_f$, 
et que l'ouvert affine $\widetilde{(X,x)}$ de $\P_{\hrt x/\widetilde k}$ est 
de la forme $\P_{\hrt x/\widetilde k}\{\phi_i\}_{i\in I}$ pour un certain
sous-ensemble~$I$ 
de $\{1,\ldots, m\}$. 

Soit $P$ le cône rationnel $\phi_\trop(\zr{\hrt x}{\widetilde k})$. 
Comme la famille $(\phi_i)$ contient une base de transcendance 
de $\hrt x$ sur $\widetilde k$, le cône $P$ est purement de dimension~$n$
(théorème \ref{theo.foncteurs-polyedraux} (1b)). 
Soit $Q$ l'intersection de $P$
et du cône rationnel
de $\Gamma^n$ défini par les inégalités
$x_i\leq 1$ pour $i\in I$.
Comme $\widetilde{(X,x)}$
est égal à $\P_{\hrt x/\widetilde k}\{\phi_i\}_{i\in I}$,
on a $\widetilde{(X,x)}=\phi_\trop^{-1}(Q)$
et $\phi_\trop(\widetilde{(X,x)})=Q$. 
Par conséquent, 
$f_\trop(X,x)=f_\trop(x)\cdot Q(\R_+^\times)$ 
(remarque \ref{rem-description-trop-germ}
). 

Le cône rationnel $P$ peut être décrit par une condition de la forme
$\bigvee_j \bigwedge_\ell \psi_{j\ell}\leq 1$ où chaque $\psi_{j\ell}$ est un monôme en les $x_i$
à exposants rationnels, avec coefficient $1$. Pour tout $j$, notons $C_j$
le sous-cône rationnel de $P$ décrit par la condition $\bigwedge_\ell \psi_{j\ell}\leq 1$. 
On a $P=\bigcup_j C_j$ ; comme $P$ est purement de dimension~$n$ on peut retirer de l'écriture
tous les $C_j$ de dimension~$<n$, et donc supposer que $\dim C_j=n$ pour tout~$j$. 
On a alors $Q=\bigcup_j Q\cap C_j$. Il suffit pour conclure de montrer que $f_\trop(\Sigma_f,x)$
contient $(Q(\R_+^\times), x)$, c'est-à-dire
qu'il contient  $((Q\cap C_j)(\R_+^\times), x)$ pour tout $j$. 
On fixe donc $j$ ; on écrit $C$ au lieu de $C_j$ et l'on pose $D=Q\cap C$.

Posons $\mathscr U=\phi_\trop^{-1}(C)$
et $\mathscr V=\phi_\trop^{-1}(D)$ ; par construction, $\mathscr U$ et 
$\mathscr V$ sont deux ouverts quasi-compacts 
de $\zr{\hrt x}{\widetilde k}$, et 
$\mathscr V= \mathscr U\cap \widetilde{(X,x)}$. 
Notons aussi 
que comme $C\subset P$, on a $\phi_\trop(\mathscr U)=C$. 

L'égalité 
$\mathscr V= \mathscr U\cap \widetilde{(X,x)}$
entraîne que $\mathscr V=\widetilde{(V,x)}$ où
$(V,x)$ est le domaine $(f_\trop^{-1}(C),x)$
de $(X,x)$. 
Quitte à renuméroter les $f_i$
(et les $\phi_i$) on peut supposer que la restriction à 
$C$
de $(x_i)_{1\leq i\leq n}$ est injective. 
Notons $C'$ et $D'$ les images de $C$ et $D$ dans
$\Gamma^n$ par la projection sur les $n$ premières
coordonnées ; la projection en question induit
par construction des bijections $C\simeq C'$ et 
$D\simeq D'$. 
Soit $\pi \colon (V,x)\to \gm^n$ 
le morphisme défini par les $f_i$ pour $1\leq i\leq n$, et soit~$y$ l'image 
de $x$. 

Posons
$\phi_{\leq n}=(\phi_i)_{1\leq i\leq n}$. 
L'image de $\mathscr U$ par $\phi_{\leq n,\trop}$
est égale à~$C'$, donc est de dimension~$n$. 
Ceci entraîne que $\phi_{\leq n}$
est de degré de transcendance au moins~$n$ sur~$\widetilde k$,
et donc exactement égal à~$n$; 
par conséquent, $y$ appartient au squelette standard~$S_n$ de $\gm^n$.
Le corps gradué $\hrt y$ est égal à $\widetilde k(\phi_{\leq n})$. 

Notons $\mathscr U'$ et~$\mathscr V'$ les images de~$\mathscr U$ 
et~$\mathscr V$ dans~$\P_{\hrt y/\widetilde k}$ ;
ce sont des ouverts quasi-compacts
et non vides 
de~$\P_{\hrt y/\widetilde k}$
d'après le théorème~7.2.5 de~\cite{ducros2018}
et l'on a $\mathscr V'\subset \mathscr U'$. 
Si $v$ est un élément de $\mathscr V'$
alors $\phi_{\leq n,\trop}(v)$ appartient à $D'$. 
Inversement, soit $u$ un élément de $\mathscr U'$ tel que
$\phi_{\leq n,\trop}(u)\in D'$. Choisissons un antécédent
$\omega$ de $u$ sur $\mathscr U$. Par construction, 
$\phi_\trop(\omega)$ appartient à l'image réciproque 
de $D'$ sur $C$ par la projection sur les $n$ premières coordonnées, 
qui n'est autre que $D$ puisque cette projection induit une bijection
$C\simeq C'$ ; par conséquent $\omega$ appartient 
à $\phi_\trop^{-1}(D)$, c'est-à-dire à $\mathscr V$ et $u$ appartient
donc à $\mathscr V'$. Il s'ensuit que
$\mathscr V'=\mathscr U'\cap \phi_{\leq n,\trop}^{-1}(D')$.

Soit $(U',y)$ et $(V',y)$ les germes de domaines
analytiques de $\gm^n$ tels que 
$\widetilde{(U',y)}=\mathscr U'$ et 
$\widetilde{(V',y)}=\mathscr  V'$. 

Dans les constructions ci-dessus, les germes 
$(U,x), (V,x), (U',y)$ et $(V',y)$ sont bien définis, mais leurs représentants
$U, V, U'$ et $V'$ sont arbitraires. Nous allons 
expliquer pourquoi l'on peut choisir $V$ et $V'$ compacts et tels
que $\pi$ induise un morphisme surjectif et plat de
$V$ sur $V'$. On part  pour ce faire de deux choix arbitraires de $V$ et $V'$ ; 
dans ce qui suit, «restreindre» $V$ (resp. $V'$) signifiera le remplacer
par un voisinage analytique de $x$ dans $V$ (resp. de $y$ dans $V'$), ce qui ne
modifiera pas les germes en jeu. 
Comme $d_k(y)=\dim \gm^n=n$ et comme $\gm^n$ est réduit, le  morphisme
$\pi \colon V\to \gm^n$ est plat en tout antécédent de $y$, et en particulier
en $x$. Le lieu de platitude de $\pi$ étant un ouvert (de Zariski) 
de $V$, on peut
restreindre $V$ de sorte que $\pi\colon V\to 
\gm^n$ soit plat.
Le germe $(V',y)$ est le plus petit germe de domaine par lequel~$\pi|_{(V,x)}
\colon 
(V,x)\to (\gm^n,y)$ 
se factorise
(théorème~7.3.1~(1) de~\cite{ducros2018}). 
Puisque~$\pi|_{(V,x)}$ se factorise par $(V',y)$
on peut, quitte à restreindre $V$, supposer que $V$
est compact et que $\pi(V)\subset V'$. Comme $\pi$ est plat et que
$V$ est compact, l'image 
$\pi(V)$ est un domaine analytique compact de
$\gm^n$ (\cite{ducros2018}, théorème~9.2.1). 
Le morphisme $\pi|_{(V,x)}$ se factorise alors
par $(\pi(V),y)$ ; par minimalité de 
$(V',y)$, il vient 
$(\pi(V),y)=(V',y)$. Quitte à remplacer
$V'$ par $\pi(V)$ on peut donc bien supposer
que $\pi$ induit un morphisme plat et surjectif de $V$
sur $V'$.

Comme
$\phi_\trop(\mathscr U)=C$ on a $\phi_{\leq n,\trop}(\mathscr U')
=C'$.
Comme $C'$ est purement de dimension~$n$, 
il résulte du
lemme \ref{lem-puredim-squelette}
que pour tout
$H\in \mathsf D$ et tout $h\in H^n$, 
la valuation $\eta_h$ de $\hrt y$ appartient 
à $\mathscr U'$ si et seulement si $h\in C'(H)$. En 
vertu de lemme~\ref{lem-squelette-tilde}
on a alors 
$\trop(U'\cap S(\gm^n),y)=(\trop(y)\cdot C'(\R_+^\times),\trop(y))$. 

Comme $\mathscr V'=\mathscr U'\cap \phi_{\leq n,\trop}^{-1}(D')$,
on a $(V',y)=(U',y)\cap f_\trop^{-1}(f_\trop(y)
\cdot D'(\R_+^\times)$, si bien
que $\trop(V'\cap S(\gm^n),y)= (\trop(y)\cdot D'(R_+^\times),\trop(y))$. 
Il existe donc un voisinage $\Omega'$ de $\trop(y)$
dans $\trop(y)\cdot D'(R_+^\times)$ qui est contenu dans 
$\trop(V'\cap S_n)$. 

Soit $\Omega$
l'image réciproque de $\Omega'$ sur $f_\trop(x)\cdot D(\R_+^\times)$
par la projection sur les
$n$ premières coordonnées. 
Soit $t\in \Omega$ et soit $t'$ son image dans $\Omega'$. 
Le point $t'$ appartient 
par construction à $\trop(V'\cap S(\gm^n))$.
Soit $v'\in V'\cap S(\gm^n))$ un point tel que $\trop(v')=t'$.
Comme $\pi(V)=V'$, il existe $v\in V$ tel que $\pi(v)=v'$.
Par définition de~$\Sigma_f$, cet ensemble contient
$\pi^{-1}(S(\gm^n))$, de sorte que l'on a $v\in\Sigma_f$.
Par ailleurs, $f_\trop(v)\in f_\trop(x)\cdot D(\R_+^n)$ et sa projection
sur $(\R_+^\times)^n$ coïncide avec~$t'$;
puisque cette projection est injective sur~$f_\trop(x)\cdot D(\R_+^n)$, on a $f_\trop(v)=t$.
\end{proof}

\begin{theo}\label{theo.riche}
Soit $X$ un espace $k$-analytique purement de dimension~$n$, 
compact, et soit $f$ un moment sur $X$. 
Considérons~$X$ comme un espace analytique de niveau~$n$.
Il existe une partie squelettique compacte~$\Sigma$ 
de~$X$ telle que $f_\trop(\Sigma)=f_\trop(X)$. 
\end{theo}
\begin{proof}
L'assertion est G-locale sur~$X$; on peut donc supposer
que $X$ est affinoïde.
Soit $x\in X$; d'après la proposition~\ref{prop-trop-pt-abhyankar},
il existe un point~$y_x$ de~$X$ tel que $d_k(y_x)=n$
et tel que les germes $f_\trop(X,x)$ et $f_\trop(X,y_x)$ coïncident.
En particulier, il existe un voisinage analytique compact~$V_x$ de~$x$
et un voisinage analytique compact~$V'_x$ de~$y_x$ tels que 
$f_\trop(V_x)=f_\trop(V'_x)$.
D'après la proposition~\ref{prop-tropicalisation-squelette},
il existe une partie squelettique~$\Sigma_x$ contenant~$y_x$
telle que les germes $f_\trop(X,y_x)$ et $f_\trop(\Sigma_x,y_x)$
coïncident.
Quitte à restreindre~$V'_x$ et à remplacer~$V_x$
par $V_x\cap f_\trop^{-1}(f_\trop(V'_x))$,  on peut supposer
que $\Sigma_x$ est compact et 
que $f_\trop(V_x)= f_\trop(V'_x)=f_\trop(\Sigma_x)$.
Comme $X$ est compact,
il existe une partie finie~$S$ de~$X$ telle que la famille $(V_x)_{x\in S}$
recouvre~$X$. 
Posons $\Sigma= \bigcup_{x\in S} \Sigma_x$;
c'est une partie squelettique compacte de~$X$ et l'on a 
\[ f_\trop(X)=\bigcup_{x\in S} f_\trop(V_x) = \bigcup_{x\in S} f_\trop (\Sigma_x) = f_\trop(\Sigma). \]
Cela conclut la démonstration du théorème.
\end{proof}

\begin{coro}\label{coro.riche}
Soit $X$ un espace $k$-analytique topologiquement
séparé et purement de dimension~$n$;
considérons-le comme un espace analytique de niveau~$n$.
Alors $X$ est riche (définition~\ref{defi.riche}).
\end{coro}
\begin{proof}
Soit $V$ un domaine compact de~$X$ et soit $f$ une carte G-tropicale
déinie sur~$V$; il s'agit de démontrer qu'il existe une partie
squelettique compacte~$\Sigma$ de~$V$ telle que $f(\Sigma)=f(V)$.
Soit $(V_i)_{i\in I}$ un G-recouvrement fini de~$V$
par des domaines analytiques compacts tels que,
pour tout $i\in I$, il existe $a_i\in\Q$ et un moment~$f_i$
sur~$V_i$ tel que $f|_{V_i}=a_i f_{i,\trop}$.
D'après le théorème~\ref{theo.riche}, il existe,
pour tout $i\in I$, une partie squelettique compacte~$\Sigma_i$
de~$V_i$ telle que $f_{i,\trop}(\Sigma_i)=f_{i,\trop}(V_i)$.
Soit $\Sigma$ la réunion de~$\Sigma_i$;
c'est une partie squelettique compacte de~$V$ et l'on a 
\[ f(\Sigma) = \bigcup_i f(\Sigma_i)
= \bigcup_i a_i f_{i,\trop}(\Sigma_i) = \bigcup_i a_i f_{i,\trop}(V_i)
 =  \bigcup_i f(V_i) = f(V). \]
Le corollaire est ainsi démontré.
\end{proof}

\section{Les espaces de Berkovich sans bord sont très riches}

\begin{theo}\label{theo.berk-tres-riche}
Soit $X$ un espace $k$-analytique sans bord, topologiquement
séparé et purement de dimension~$n$;
considérons-le comme un espace analytique de niveau~$n$.
Alors $X$ est très riche (définition~\ref{defi.riche}).
\end{theo}
\begin{proof}
D'après le corollaire~\ref{coro.riche}, l'espace~$X$ est riche ; 
il reste à démontrer
que toute partie squelettique compacte~$\Sigma$ de~$X$ est contenue
dans une partie squelettique purement de dimension~$n$.

Soit $x$ un point de~$\Sigma$.
Le germe $(X,x)$ est sans bord, donc presque bon ;
d'après la proposition~\ref{prop.parties-polyedrales-moments},
il existe un voisinage ouvert~$U_x$ de~$x$ 
et un moment~$f_x\colon U_x\to\gm^n$
tel que $\Sigma\cap U_x\subset \Sigma_{f_x}$.
Considérons un voisinage analytique compact~$V_x$ de~$x$ dans~$U_x$.
Démontrons que $x\in (\Sigma_{f_x} \cap V_x)^{(n)}$.

Comme $\Sigma$ est compact, il existe une partie finie~$S$ de~$\Sigma$
telle que les~$\mathring{V_x}$, pour $x\in S$, recouvrent~$\Sigma$.
Soit $\Tau$ la réunion des $(\Sigma_{f_s}\cap V_s)^{(n)}$, pour $s\in S$.
c'est une partie squelettique compacte de~$X$, 
purement de dimension~$n$.

Démontrons que $\Sigma\subset \Tau$. 
Soit $x\in \Sigma$; soit $s\in S$ tel que $x\in \mathring{V_s}$.
Comme $X$ est sans bord, l'ouvert~$\mathring V_s$ est sans bord,
donc $\Sigma_{f_s}\cap \mathring{V_s}$ est purement
de dimension~$n$ 
(corollaire~\ref{coro.Sigmaf=Sigmafn}),
si bien que $x$ appartient à $(\Sigma_{f_s}\cap V_s)^{(n)}$.
A fortiori, $x\in \Tau$.
\end{proof}

\begin{rema}
On ne peut pas, en général, enlever l'hypothèse que $X$ est sans bord.
Supposons par exemple qu'il existe $r\notin\abs{k^\times}^\Q$
et considérons le domaine analytique de~$\A^1_k$ 
défini par $\abs T=r$. C'est un espace $k$-analytique de dimension~$1$,
réduit au singleton~$\{\eta_r\}$, qui est aussi son bord.
Il ne contient aucun squelette de dimension~$1$, et n'est donc
pas très riche.
\end{rema}

\begin{rema}
La structure des courbes $k$-analytiques
(pour laquelle nous renvoyons au chapitre~4 de \cite{berkovich1990}, 
et à la présentation plus générale de~\cite{Ducros-2024}) entraîne
que toute courbe $k$-analytique ne comportant aucun point de type~(3) isolé
est très riche. En effet, soit $X$ une telle courbe et soit $x\in X$
un point tel que $d_k(x)=1$. Cela entraîne que $X$ est de type~(2)
ou~(3).
Si $x$ est de type~(2), il existe une branche aboutissant à~$x$;
il en est de même si $x$ est de type~(3) puisque, par hypothèse,
il n'est pas isolé.

En particulier, pour tout~$r>0$, le disque de rayon~$r$ de~$\A^1_k$
est très riche.
\end{rema}

\begin{rema}
Nous nous attendons à ce que le polydisque unité de~$\A^n_k$
soit très riche. 

En revanche, si $r\notin\abs{k^\times}\Q$,
le bidisque~$X$ de birayon $(r,1/r)$ n'est pas très riche.
Considérons en effet l'ouvert~$U$ de~$X$ défini
par $\abs{T_1T_2-1}<1$; nous allons
démontrer qu'il est non vide et contenu dans la fibre
de la première projection en~$\eta_r$,
qui est une courbe $\hr{\eta_r}$-analytique,
et ne contient donc aucune partie squelettique de dimension~$2$.

Soit $x\in U$. Comme $\abs{T_1(x)T_2(x)-1}<1$,
on a $\abs{T_1(x)}\abs{T_2(x)}=1$; par définition de~$X$,
on a $\abs{T_1(x)}\leq r$ et $\abs{T_2(x)}\leq 1/r$,
si bien que $\abs{T_1(x)=r}$.
Alors $y=p_1(x)$ est un point de~$\A^1_k$ tel que $\abs{T_1(y)}=r$;
comme $r\notin\abs{k^\times}^\Q$, on a $y=\eta_r$.

Il reste à vérifier que $U$ n'est pas vide.
Mais $p_1^{-1}(\eta_r)$ est le  domaine de $\A^1_{\hr{\eta_r}}$,
de coordonnée~$T_2$,
défini par $\abs{T_2}\leq 1/r$. Le point $\hr{\eta_r}$-rationnel
de ce domaine défini par $T_2=T_1(\eta_r)^{-1}$ appartient à~$U$.
\end{rema}

\chapter{Compléments de géométrie analytique}

\section{Anneaux G-locaux}

\subsection{}
Soit $X$ un espace $k$-analytique et soit $x$ un point de $X$. 
Nous avons défini plus haut la dimension centrale du germe $(X,x)$ 
(\ref{ss.dimcent}). Celle-ci possède une variante G-topologique : 
la dimension centrale du G-germe $(X_\groth, x)$, notée 
$\dim_{\mathrm c}(X_\groth ,x)$, et définie 
comme 
le minimum des dimensions de $\overline{\{x\}}^{V_{\mathrm{Zar}}}$ 
lorsque $V$ parcourt l'ensemble des domaines analytiques de~$X$ contenant~$x$. 

On a de manière évidente
\[d_k(x)\leq \dim_{\mathrm c}(X_\groth,x)
\leq \dim\overline{\{x\}}^{X_{\mathrm{Zar}}}\leq \dim_x X \;; \]
en particulier $\dim_{\mathrm c}(X_\groth,x)=\dim_x X$ dès que $d_k(x)=n$. 

Soit $V$ un domaine analytique de~$X$ contenant~$x$. 
L'ensemble des domaines analytiques de $V$ contenant $x$ est cofinal dans l'ensemble
des domaines analytiques de~$X$ contenant~$x$ ; par conséquent
$\dim_{\mathrm c}(V_\groth,x)=\dim_{\mathrm c}(X_\groth,x)$. 

Soit $Z$ un sous-espace analytique fermé de $X$ contenant $x$. L'ensemble des domaines analytiques de $Z$ contenant $x$
de la forme $V\cap Z$, où $V$ est un domaine analytique de $X$ contenant $x$, est cofinal dans l'ensemble
des domaines analytiques de $X$ contenant $x$. Par conséquent 
Par conséquent $\dim_{\mathrm c}(Z_\groth,x)$ est le minimum des 
$\dim \overline{\{x\}}^{(V\cap Z)_{\mathrm{Zar}}}$ où $V$
parcourt l'ensemble des domaines analytiques de $X$ contenant $x$. 
Mais pour un tel $V$ on a $\overline{\{x\}}^{(V\cap Z)_{\mathrm{Zar}}}
=\overline{\{x\}}^{V_{\mathrm{Zar}}}$, si bien que 
$\dim_{\mathrm c}(Z_\groth,x)=\dim_{\mathrm c}(X_\groth,x)$.

\subsection{}\label{ss-sorites-xgen}
Nous noterons $X^\gen$ l'ensemble des points $x$ de $X$ tels que $\dim_{\mathrm c}(X_\groth,x)=\dim_x X$.
Cet ensemble contient $\{x\in X,d_k(x)=\dim_x X\}$. 

Si $V$ est un domaine analytique de $X$ on a $\dim_{\mathrm c}(V_\groth,x)=\dim_{\mathrm c}(X_\groth,x)$
et $\dim_x V=\dim_x X$ pour tout $x\in V$, si bien que $V^\gen=X^\gen \cap V$. 

Si $Z$ est un sous-espace analytique fermé de $X$ on a $\dim_{\mathrm c}(Z_\groth,x)=\dim_{\mathrm c}(X_\groth,x)$
pour tout $x\in Z$, si bien que $Z^\gen$ est l'ensemble des points $x$ de $Z$ tels que $\dim_{\mathrm c}(Z_\groth,x)
=\dim_x Z$ ; cet ensemble ne dépend que du fermé de Zariski sous-jacent à $Z$, et pas de la structure de sous-espace
analytique fermé choisie sur $Z$.

On déduit du lemme \ref{lemm-adherence-point} que $x\in X^\gen$ si et seulement si $\overline{\{x\}}^{V_{\mathrm{Zar}}}$
est une composante irréductible de $V$ pour tout domaine analytique $V$ de $X$ contenant $x$. 
L'appartenance à~$X^\gen$ est donc une propriété de généricité universelle pour la topologie de Zariski, d'où la notation 
choisie. 

Soit $(X_i)$la famille des composantes irréductibles de $X$ ; pour tout $i$, soit $U_i$ l'ouvert de Zariski 
$X_i\setminus \bigcup_{j\neq i}X_j$ de $X$, et soit $U$ la réunion des $U_i$ (qui sont deux à deux disjoints). 
Si $x\in X^\gen$ alors  $\overline{\{x\}}^{X_{\mathrm{Zar}}}$ est une composante irréductible de $X$, et c'est la seule
composante irréductible de $X$ contenant $x$ (lemme  \ref{lemm-adherence-point}). Il s'ensuit que $X^\gen=U^\gen=\bigcup
U_i^\gen$ (cette dernière union étant disjointe). Pour tout $i$, l'ouvert non vide $U_i$ de l'espace irréductible $X_i$
est lui-même irréductible (\cite{Ducros-2021}, lemme 1.11), et $U_i^\gen$ est non vide : il contient par exemple
tous les points $x$ de $X_i$ tels que $d_k(x)=\dim X_i$. Notons aussi que $U_i^\gen=X_i^\gen$ puisque tout point de
$X_i^\gen$ est dense dans $X_i$ pour la topologie de Zariski.

\begin{defi}\label{defi.anneau-G-local}
Soit $X$ un espace $k$-analytique et soit $x$ un point de~$X$. 
L'anneau G-local de~$X$ en~$x$ 
est la colimite
des $k$-algèbres~$\mathscr O_X(V)$ 
où $V$ parcourt l'ensemble des domaines analytiques de $X$ contenant $x$. 
On le note $\mathscr O_{X_\groth,x}$.
\index{anneau G-local d'un espace analytique}
\end{defi}

C'est un anneau local hensélien et excellent
(\cite{ducros2024}, proposition~6.7, (1)). 

Définissons la dimension centrale du germe~$(X_\groth, x)$
comme 
le minimum des dimensions de $\overline{\{x\}}^{V_{\mathrm{Zar}}}$ 
lorsque $V$ parcourt l'ensemble des domaines analytiques de~$X$ contenant~$x$;
on la note~$\dim_{\mathrm c}(X_\groth, x)$.
Alors, on a 
\[  \dim_{\mathrm c}(X_\groth, x)+\dim_{\mathrm{Krull}}(\mathscr O_{X_\groth, x})=\dim_x X \]
(\cite{ducros2024}, proposition~6.7, (1)). 
En particulier, l'anneau local $\mathscr O_{X_\groth, x}$ est artinien si et seulement si
$x\in X^\gen$. 

Pour tout bon domaine analytique~$V$ de~$X$ contenant~$x$,
le morphisme naturel de $\mathscr O_{V,x}$ 
vers~$\mathscr O_{X_\groth, x}$ est régulier 
(\cite{ducros2024}, proposition~6.7, (3)). 

\subsection{}
Soit $Y$ un espace analytique défini sur une extension de $k$,  soit $f$ un $k$-morphisme de $Y$ vers $X$ et soit 
$y$ un antécédent de $X$. Le morphisme $f^*\colon \mathscr O_X\to f_*\mathscr O_Y$ induit un morphisme de
$\mathscr O_{X_\groth,x}$ vers $\mathscr O_{Y_\groth, y}$. 
Si $Y$ est $k$-analytique, le morphisme $f$ est plat en~$y$ 
si et seulement si $\mathscr O_{Y_\groth, y}$ 
est une $\mathscr O_{X_\groth,x}$-algèbre plate
(\cite{ducros2024}, théorème 7.5). 

\subsection{}\label{ss-XKG-plat}
Soit $K$ une extension complète de $k$, soit $y$ un point de $X_K$ et soit 
$x$ son image sur $X$. L'anneau $\mathscr O_{X_{K,\groth}, y}$ est alors plat
sur $\mathscr O_{X_\groth,x}$. 
En effet, soit $V$ un bon domaine analytique de $X$ contenant $x$. Alors
$\mathscr O_{V_K,y}$ est plat sur $\mathscr O_{V,x}$, si bien que par composition
$\mathscr O_{X_{K,\groth},x}$ est plats sur $\mathscr O_{V,x}$. L'énoncé
voulu s'obtient en passant à la colimite sur $V$.

\subsection{}
Soit $X$ un espace $k$-analytique et soit $F$ un faisceau cohérent sur $X$. 
Nous noterons
$F_{X_\groth,x}$ la colimite
des $F(V)$ où $V$ parcourt l'ensemble des domaines
analytiques de $X$ contenant $x$. C'est un $\mathscr O_{X_\groth,x}$-module
de type fini, qui s'identifie à 
$\mathscr O_{V_\groth,x}\otimes_{\mathscr O_{V,x}}F_{V,x}$. 

Soit $\mathsf P$ et $\mathsf Q$ 
deux propriétés algébriques comme en 
\ref{ss-liste-proprietes}. 
Puisque le morphisme naturel de $\mathscr O_{V,x}$ vers
$\mathscr O_{V_\groth,x}$ est régulier, l'espace $X$ (\resp le faisceau cohérent
$\mathscr F$) satisfait $\mathsf P$ (\resp~$\mathsf Q$) en~$x$ 
si et seulement si
l'anneau local noethérien $\mathscr O_{V_\groth,x}$ (\resp le
$\mathscr O_{V_\groth,x}$-module de type fini $\mathscr F_{V_\groth,x}$) satisfait~$\mathsf P$ (\resp~$\mathsf Q$). 

\begin{lemm}\label{lemme-domaine-cofinal}
Soit $f\colon Y\to X$ un morphisme fini et soit $x$ un point de $X$
possédant un unique antécédent $y$ sur $Y$. 
Soit $V$ un domaine analytique de $Y$ contenant $y$. Il existe un domaine affinoïde $U$ de $X$
contenant $x$ tel que $f^{-1}(U)\subset V$. 
\end{lemm}

\begin{proof}
Pour démontrer le lemme, on peut remplacer~$X$
par n'importe lequel de ses domaines analytiques contenant~$x$, 
ce qui permet de le supposer séparé. 
Choisissons un pré-modèle~$\mathsf X$ de~$\widetilde{(X,x)}$ 
de dimension égale à~$d_k(x)$. 
Comme $f$ est sans bord, $\mathsf X$ est aussi un pré-modèle 
de~$\widetilde{(Y,y)}$. Il existe alors un pré-modèle 
$\mathsf Y$ de $\widetilde{(Y,y)}$, propre et dominant sur~$\mathsf X$, 
et un ouvert de Zariski~$\mathsf V$ de~$\mathsf Y$ 
qui est un pré-modèle de $\widetilde{(V,y)}$. 
Puisque $f$ est fini, on a $d_k(y)=d_k(x)$; 
on en déduit l'égalité $\dim(\mathsf Y)=\dim(\mathsf X)=d_k(x)$.
Le complémentaire~$\mathsf Z$ de~$\mathsf V$ dans~$\mathsf Y$ 
est un fermé de Zariski strict de~$\mathsf Y$ 
et est donc de dimension $<d_k(x)$. 
Par propreté, son image $\mathsf T$ sur~$\mathsf X$ est un fermé de Zariski, 
lui aussi de dimension $<d_k(x)$ ; 
en particulier, il est distinct de~$\mathsf X$. 
Soit $\mathsf U$ un ouvert affine non vide de~$\mathsf X$ 
ne rencontrant pas~$\mathsf T$ et 
soit~$\mathsf W$ son image réciproque sur~$\mathsf Y$. 
Par construction, $\mathsf W\subset \mathsf V$. 
Soit~$U$ un domaine affinoïde de~$X$ contenant~$x$ tel que 
$\widetilde{(U,x)} =\zr{\hrt x}{\widetilde k}\{\mathsf U\}$.  
L'image réciproque~$W$ de~$U$ sur~$Y$ vérifie l'égalité 
$\widetilde{(W,x)} =\zr{\hrt y}{\widetilde k}\{\mathsf W\}$, 
si bien qu'on a une inclusion de germes $(W,y)\subset (V,y)$. 
Comme le morphisme de~$W$ dans~$U$ induit par~$f$ est topologiquement
propre, 
il en résulte qu'il existe un voisinage affinoïde de~$x$ 
dans~$U$ dont l'image réciproque par~$f$ est contenue dans~$V$. 
\end{proof}

\begin{coro}\label{coro-quasifini-cofinal}
Soit $f\colon Y\to X$ un morphisme d'espaces $k$-analytiques, 
soit~$y$ un point en lequel $f$ est quasi-fini 
et soit~$x$ l'image de~$y$ sur $X$. 
Il existe un domaine affinoïde~$U$ de~$X$ contenant~$x$
et un domaine affinoïde~$V$ de $f^{-1}(U)$ contenant~$y$ qui est fini sur~$U$.
\end{coro}

\begin{proof}
Pour démontrer le corollaire, on peut remplacer $(X,Y)$ 
par n'importe lequel couple $(X',Y')$,
où $X'$ est un domaine analytique de~$X$ contenant~$x$ et 
$Y'$ est un domaine analytique de $f^{-1}(X')$ contenant~$Y$.
Ceci permet de se ramener au cas où $X$ et $Y$ sont affinoïdes. 
D'après la version analytique du \textit{Main Theorem} de Zariski
(\cite{ducros2007}, théorème 3.2),
on peut alors supposer, 
quitte à remplacer $X$ et $Y$ par des voisinages affinoïdes convenables 
de~$x$ et~$y$ respectivement, 
qu'il existe une factorisation 
\[\begin{tikzcd}
Y\ar[r,"g"]\ar[rrr, bend left, "f"]&W\ar[r, hook]&T\ar[r,"h"]&X\end{tikzcd}
\]
de~$f$, 
où $h$ est étale, $W$ est un domaine affinoïde de $T$, où $g$ est fini. 
Soit~$t$ l'image de~$y$ sur~$T$. 
Il existe un voisinage affinoïde~$T'$ de~$t$ dans~$T$
fini étale sur un voisinage affinoïde~$X'$ de~$x$ dans~$X$, 
et tel que $t$~soit le seul antécédent de~$x$ sur~$T'$. 
Posons $W'=W\cap T'$ et $Y'=Y\times_T T'$.
On obtient par changement de base le diagramme
\[\begin{tikzcd}
Y'\ar[r,"g'"]\ar[rrr, bend left, "f'"]&W'\ar[r, hook]&T'\ar[r,"h'"]&X'\end{tikzcd}
\]
où $g'$ est fini et où $h'$ est fini étale. 
Le lemme précédent assure l'existence d'un domaine affinoïde~$U$
de~$X'$ contenant~$x$ tel que $(h')^{-1}(U)\subset W'$. 
Posons $V=(f')^{-1}(U)$. Il résulte de l'inclusion $(h')^{-1}(U)\subset W'$
et de la finitude de~$g'$ que $V$ est un domaine affinoïde de~$Y'$
(et donc de~$Y$) fini sur $(h')^{-1}(U)$. Par finitude de~$h'$, 
le domaine~$V$ est alors fini sur~$U$, et il contient~$y$ par construction. 
\end{proof}

\begin{prop}\label{prop-gfini}
Soit $f\colon Y\to X$ un morphisme d'espaces $k$-analytiques. 
Soit~$y$ un point de~$Y$ et soit~$x$ son image sur~$X$. 
Supposons que $f$ est quasi-fini en $y$.
\begin{enumerate}
\item On a $\dim_{\mathrm c}(X_\groth, x)=\dim_{\mathrm c}(Y_\groth,y)$. 
\item Le morphisme $f^*$ de $\mathscr O_{X_\groth, x}$ vers $\mathscr O_{Y_\groth, y}$ 
est fini. 
\item 
Si $f$ est fini en $y$ et si les germes $(X,x)$ et $(Y,y)$ sont bons,
le morphisme canonique  \[\mathscr O_{Y,y}\otimes_{\mathscr O_{X,x}}\mathscr O_{X_\groth, x}
\to \mathscr O_{Y_\groth, y}\]
est un isomorphisme. 
\item
Pour tout $x\in X$, on a
$\mathscr O_{X_{\red, \groth}, x}=(\mathscr O_{X_\groth,x})_\red$.
\end{enumerate}
\end{prop}

\begin{proof}
Par définition de la dimension centrale 
et des anneaux G-locaux, on peut pour démontrer~\emph a) et~\emph b),
remplacer~$X$ par un domaine affinoïde~$V$ de~$X$ contenant~$x$
et~$Y$ par un domaine affinoïde de $Y\times_X U$ contenant~$y$.
Compte tenu du corollaire \ref{coro-quasifini-cofinal},
on peut supposer que~$X$ et $Y$ sont affinoïdes, 
que $f$ est fini et que $y$ est le seul antécédent de~$x$ sur~$Y$. 

De même, pour démontrer~\emph c), on peut remplacer
remplacer~$X$ 
un voisinage affinoïde~$U$ de~$x$ dans~$X$, 
et $Y$ par un voisinage affinoïde de~$y$ dans $Y\times_X U$. 
On peut donc également supposer que~$X$ et $Y$ sont affinoïdes, 
que $f$ est fini et que $y$ est le seul antécédent de~$x$ sur~$Y$. 

L'ensemble des domaines affinoïdes de~$Y$ de la forme $f^{-1}(U)$,
où $U$ est un domaine affinoïde de~$X$ contenant~$x$,
est cofinal dans l'ensemble des domaines analytiques de~$Y$ 
contenant~$y$ (lemme \ref{lemme-domaine-cofinal}). 
En conséquence l'anneau local 
$\mathscr O_{Y_\groth,y}$ est la colimite des anneaux $\mathscr O_Y(f^{-1}(U))$ où $U$ parcourt
l'ensemble des domaines affinoïdes de~$X$ contenant~$x$. 
Pour un tel~$U$, on a l'égalité
$\mathscr O_Y(f^{-1}(U))=\mathscr O_Y(Y)\otimes_{\mathscr O_X(X)}\mathscr O_X(U)$, si bien que 
$\mathscr O_{Y_\groth, y}$ s'identifie à $\mathscr O_Y(Y)\otimes_{\mathscr O_X(X)}\mathscr O_{X_\groth,x}$, d'où~\emph b). 
Alors,
\[ \mathscr O_{Y,y}\otimes_{\mathscr O_{X,x}} \mathscr O_{X_\groth,x}
= 
\mathscr O_Y(Y)\otimes_{\mathscr O_X(X)}\mathscr O_{X,x}
\otimes_{\mathscr O_{X,x}} \mathscr O_{X_\groth,x}
= 
\mathscr O_Y(Y)\otimes_{\mathscr O_X(X)}\mathscr O_{X_\groth,x}
=
\mathscr O_{Y_\groth,y},\]
ce qui démontre~\emph c).

Prouvons \emph a). Soit $U$ un domaine analytique de $X$ contenant $x$ tel que l'adhérence de Zariski $Z$
de $x$ dans $U$ soit de dimension $\dim_{\mathrm c}(X_\groth, x)$. Alors $f^{-1}(Z)$ est un fermé de Zariski 
de $f^{-1}(U)$ contenant~$y$ et fini sur~$Z$.
Alors $f(f^{-1}(Z))$ est un fermé de Zariski de~$U$ qui contient~$x$;
il est donc égal à~$Z$.
Par suite, la dimension de~$f^{-1}(Z)$ est donc égale à $\dim Z$, si bien que
$\dim_{\mathrm c}(Y_\groth, y)\leq \dim_{\mathrm c}(X_\groth, x)$. 

Réciproquement, soit $V$ un domaine analytique de $Y$ contenant $y$ tel que l'adhérence de Zariski $T$ de $y$
dans $V$ soit de dimension $\dim_{\mathrm c}(Y_\groth, y)$. Par le lemme \ref{lemme-domaine-cofinal} 
il existe un domaine affinoïde $U$ de $X$ contenant $x$ tel que $f^{-1}(U)\subset V$. L'intersection 
 $T\cap f^{-1}(U)$ est encore de dimension $\dim_{\mathrm c}(Y_\groth, y)$ et son image par
 le morphisme fini $f$
 est un fermé de Zariski de $U$ contenant $x$ et de dimension $\dim_{\mathrm c}(Y_\groth, y)$. 
 Il vient $\dim_{\mathrm c}(X_\groth, x)\leq \dim_{\mathrm c}(Y_\groth, y)$, ce qui achève la démonstration. 

Démontrons enfin l'assertion~\emph d).
Soit $\mathscr N$ le nilradical de~$\mathscr O_X$,
de sorte que $X_\red = V(\mathscr N)$.
D'après~\emph c), on a 
$\mathscr O_{X_{\red,\groth},x}=\mathscr O_{X_\groth,x}/\mathscr N \mathscr O_{X_\groth,x}$.
L'idéal $\mathscr N \mathscr O_{X_\groth,x}$ est nilpotent;
d'autre part, comme les domaines analytiques d'un espace réduit
sont réduits, l'anneau $\mathscr O_{X_{\red,\groth},x}$ est réduit.
Il en résulte que $\mathscr N \mathscr O_{X_\groth,x}$
est le nilradical de $\mathscr O_{X_\groth,x}$,
d'où l'assertion. 
\end{proof}

\begin{coro}\label{coro-gfini}
Soit $f\colon Y\to X$ un morphisme fini d'espaces $k$-analytiques, 
soit $x$ un point de $X$
et soit $y_1,\ldots, y_r$ ses antécédents sur $Y$. 
Soit $F$ un faisceau cohérent sur $Y$. 
On a alors \[(f_*F)_{X_\groth,x}=\prod_{i=1}^r F_{Y_\groth,y_i}.\]
\end{coro}

\begin{proof}
Soit $V$ un bon domaine analytique de $X$ contenant $x$ ; posons $W=f^{-1}(V)$. Notons $H$ la restriction 
de $G$ à $W$ ; par un léger abus, notons $f_*H$ l'image directe de $H$ sous le morphisme fini de $W$ vers $V$
induit par $f$. Remarquons que $f_*H=(f_*G)|_V$. 
On a 
\begin{align*}
(f_*G)_{X_\groth,x}&=(f_*H)_{V,x}\otimes_{\mathscr O_{V,x}}\mathscr O_{X_\groth,x}\\
&=(\prod_{i=1}^r H_{W,y_i})\otimes_{\mathscr O_{V,x}}\mathscr O_{X_\groth,x}\\
&=\prod_{i=1}^r (H_{W,y_i}\otimes_{\mathscr O_{V,x}}\mathscr O_{X_\groth,x})\\
&=\prod_{i=1}^r G_{Y_\groth,y_i},
\end{align*}
où la deuxième égalité provient de l'analogue du corollaire pour les anneaux locaux classiques (voir par exemple
le lemme 4.1.15 de \cite{ducros2018}) 
et la dernière
de la proposition \ref{prop-gfini} \emph c). 

\end{proof}

\section{Longueurs et multiplicités génériques}

\begin{lemm}\label{lemme-longueur-artinien}
Soit $A$ un anneau local artinien d'idéal maximal 
$\mathfrak m$ et soit $B$ une $A$-algèbre locale artinienne plate et telle
que $\mathfrak mB$ soit l'idéal maximal de $B$. 
Soit $M$ un $A$-module de type fini. 
On a l'égalité
\[\ell_B(B\otimes_A M)=\ell_A(M).\]
\end{lemm}

\begin{proof}
Soit $\{0=M_0\subset M_1\ldots \subset M_{\\ell_A(M)}=M$ une filtration de $M$
telle que le $A-module$ $M_{i+1}/M_i$ soit isomorphe à $A/\mathfrak m$. 
Par platitude de $B$ sur $A$ cette filtration induit une filtration 
$\{0=B\otimes_A M_0\subset 
B\otimes_A M_1\subset \ldots\subset B\otimes_A M_{\ell_A(M)}$ 
de $B\otimes_A{\ell_A(M)}$, dont les quotients successifs sont tous
isomorphes comme $B$-modules à $B/\mathfrak m B$, qui est par hypothèse le corps
résiduel de $B$. En conséquence, $\ell_B(B\otimes_A M)=\ell_A(M)$. 
\end{proof}

\begin{coro}\label{coro-longueur-artinien}
Soit $A$ un anneau local artinien d'idéal maximal 
$\mathfrak m$ et soit $B$ une $A$-algèbre locale artinienne 
telle que le morphisme structural de $A$ dans $B$ soit
régulier. Soit $M$ un $A$-module de type fini. 
On a l'égalité
\[\ell_B(B\otimes_A M)=\ell_A(M).\]
\end{coro}

\begin{proof}
Comme le morphisme de $A$ vers $B$ est régulier, la $(A/\mathfrak m)$-algèbre 
artinienne $B/\mathfrak mB$ est géométriquement régulière, et a fortiori réduite. 
C'est donc un corps, ce qui signifie que $\mathfrak mB$ est l'idéal maximal de $B$. 
On conclut par une application directe du lemme précédent. 
\end{proof}

\subsection{}
Soit $X$ un espace $k$-analytique. 
Soit $x$ un point de $X^\gen$. 
L'anneau local $\mathscr O_{X_\groth, x}$ est artinien, ainsi que l'anneau local 
$\mathscr O_{V,x}$ pour tout bon domaine analytique de~$X$ contenant~$x$ (puisque
$\mathscr O_{X_\groth, x}$ est plat sur $\mathscr O_{V,x}$). 
Si $F$ est un faisceau cohérent sur $X$ nous noterons~$\ell_{X_\groth, x}(F)$
la longueur du $\mathscr O_{X_\groth,x}$-module de type fini 
$F_{X_\groth,x}$.
Si $V$ est un bon domaine
analytique de $X$ contenant $x$ nous noterons~$\ell_{V,x}$ la longueur du
$\mathscr O_{V,x}$-module de type fini~$F_{V,x}$. 

\begin{prop}\label{prop-longueur-generique}
Soit $X$ un espace $k$-analytique et soit $F$ un faisceau cohérent sur 
$X$. Soit $(X_i)$ la famille des composantes irréductibles de $X$ et pour tout $i$, soit $U_i$
l'ouvert de Zariski $X_i\setminus \bigcup_{j\neq i}X_j$ de $X$. 

\begin{enumerate}
\item Pour tout $i$, l'application $x\mapsto \ell_{X_\groth,x}(F)$
est constante sur~$U_i^\gen$. 
\item Pour tout $x\in X^\gen$ et tout bon domaine analytique~$V$ de~$X$
contenant~$x$ 
on a l'égalité
$\ell_{ X_\groth,x} (F)= \ell_{V,x} (F)$. 
\end{enumerate}
\end{prop}
Rappelons que les $U_i$ sont irréductibles, que les 
$U_i^\gen$ sont non vides, deux à deux disjoints, et qu'ils recouvrent $X^\gen$ (\ref{ss-sorites-xgen}). 

\begin{proof}
Montrons tout d'abord \emph a). Fixons $i$. 
On peut pour alléger les notations remplacer $X$ par $U_i$, c'est-à-dire supposer $X$ irréductible. 
Il s'agit alors de montrer que l'application $x\mapsto \ell_{X_\groth,x}(F)$ de $X^\gen$
dans $\N$ est constante. Notons que comme $X$ est irréductible, tout point de $X^\gen$ est dense
dans $X$ pour la topologie de Zariski.

Soit $I$ le faisceau des fonctions localement nilpotentes sur $X$. 
Faisons une remarque générale. 
Soit $x\in X^\gen$. L'anneau local artinien $\mathscr O_{X_\groth,x}/I\mathscr O_{X_\groth,x}$
est égal à $\mathscr O_{X_{\mathrm{red}, \groth},x}$, et est donc réduit. C'est en conséquence un corps, si bien 
que $I\mathscr O_{X_\groth,x}$ est l'idéal maximal de $\mathscr O_{X_\groth,x}$. 

Fixons maintenant $x_0\in X^\gen$. 
Comme $\mathscr O_{X_\groth,x_0}$ est noethérien, 
il existe $n$ tel que $I^n\mathscr O_{X_\groth,x_0}=0$.  Pour tout $i$ entre $0$ et $n-1$, 
notons $r_i$
la dimension du $\mathscr O_{X_\groth,x_0}/I\mathscr O_{X_\groth,x_0}$-espace vectoriel 
$I^i\mathscr O_{X_\groth,x_0}/
I^{i+1}F_{X_\groth,x_0}$. 

Soit $U$ le plus grand ouvert de $X$ au-dessus duquel
$I^n$ est nul et chacun des
$\mathscr O_{X_{\mathrm{red}, \groth},x}$-modules $I^i F/I^{i+1}\mathscr O_X$ est localement libre de rang $r_i$. 
C'est un ouvert de Zariski de $X$, non vide puisqu'il contient $x_0$ par construction ; par
conséquent, $X^\gen \subset U$. 
Soit $x$ un point de $X^\gen$. Puisque $x\in U$, on dispose d'une filtration 
\[0=I^nF_{X_\groth,x}\subset I^{n-1}F_{X_\groth,x}\subset \ldots
\subset IF_{X_\groth,x}\subset IF_{X_\groth,x},\]
dont le $i$-ème quotient $\mathscr I^iF_{X_\groth, x}/\mathscr I^{i+1}F_{X_\groth,x}$ est
un espace vectoriel de dimension $r_i$
sur le corps résiduel $\mathscr O_{X_\groth,x}/I\mathscr O_{X_\groth,x}$. 
Il s'ensuit que $\ell_{X_\groth,x}$ est égal à $\sum_i r_i$ pour tout $x\in X^\gen$, ce qui termine la preuve de
\emph a).

L'assertion \emph b) résulte quant à elle du fait que le morphisme
de $\mathscr O_{V,x}$ vers $\mathscr O_{X_\groth, x}$ est géométriquement régulier,
et du corollaire \ref{coro-longueur-artinien}. 
\end{proof}

\begin{defi}\label{defi-longueur-generique}
On conserve les notations de la proposition \ref{prop-longueur-generique}. 
La valeur constante de la fonction
$x\mapsto \ell_{X_\groth,x}(F)$ sur l'ensemble non vide~$U_i^\gen$ 
sera appelé la \textit{longueur générique de~$F$ le long de~$X_i$}. 
Si $F=\mathscr O_X$, nous dirons que c'est la \textit{multiplicité générique}
de~$X$ le long de~$X_i$. 
\end{defi}

\begin{prop}\label{prop.long-restr}
On conserve les notations de la proposition \ref{prop-longueur-generique}. 
Soit $V$ un domaine analytique de~$X$,
soit $W$ une composante irréductible de $V$
et soit soit~$i$ l'unique indice tel que $W\subset X_i$.
La multiplicité générique de $F|_V$ le long de~$W$
est égale à la multiplicité de~$F$ le long de~$X_i$.
\end{prop}
\begin{proof}
En effet, soit $x\in V^\gen\cap W$. Alors, $x\in U_i^\gen$.
Par définition des anneaux G-locaux,
$\mathscr O_{V_G,x}$ s'identifie à $\mathscr O_{X_G,x}$
de sorte que $\ell_{V_G,x}(F)=\ell_{X_G,x}$ pour tout
point $x\in V$. La proposition en découle.
\end{proof}

\begin{lemm}[Principe GAGA pour les longueurs génériques]
\label{lem-gaga-multiplicite}
Soit $\mathscr X$ un schéma localement de type fini sur une algèbre affinoïde $A$, soit $\mathscr Z$ une composante
irréductible de $\mathscr X$ et soit $\zeta$ son point générique. Soit $\mathscr F$ un faisceau cohérent sur $\mathscr X$. 

\begin{enumerate}
\item La longueur générique de $\mathscr F^\an$ le long de $\mathscr Z^\an$ est égale à la longueur générique de~$\mathscr F$ le long de~$\mathscr Z$, c'est-à-dire à la longueur du
$\mathscr O_{\mathscr Z\mathscr ,\zeta}$-module de type fini $\mathscr F_\zeta$. 

\item La multiplicité générique de $\mathscr X^\an$ le long de $\mathscr Z^\an$ est égale à la multiplicité générique de 
$\mathscr X$ le long de $\mathscr Z$, c'est-à-dire à la longueur de l'anneau local artinien $\mathscr O_{\mathscr Z,\zeta}$. 
\end{enumerate}
\end{lemm}

\begin{proof}
Choisissons $z\in (\mathscr Z^\an)^\gen$. Alors $z$ est Zariski-dense dans $\mathscr Z^\an$, si bien que l'image de $z$
sur $\mathscr Z$ est égale à $\zeta$. Les deux morphismes d'anneaux naturels de $\mathscr O_{\mathscr Z,\zeta}$ vers
$\mathscr O_{\mathscr Z^\an,z}$ et de $\mathscr O_{\mathscr Z^\an,z}$ vers $\mathscr O_{\mathscr Z^\an_\groth,z}$ sont réguliers, 
si bien que leur composée l'est aussi. L'assertion \emph a) découle alors du corollaire \ref{coro-longueur-artinien} ; la'assertion 
\emph b) est simplement
l'assertion \emph a) dans le cas particulier où $\mathscr F=\mathscr O_{\mathscr X}$. 
\end{proof}

\begin{lemm}[Longueur générique et extension des scalaires]\label{lemme-longueur-XK}
Soit $X$ un espace $k$-analytique et soit $K$ une extension complète de $k$. 
Soit $F$ un faisceau cohérent sur $X$. 
Soit $Z$ une composante irréductible de $X$, soit $\ell$ la longueur
générique de
$X$ le long de~$Z$, soit $\ell_K$ la longueur générique de $F_K$ le long de $Z_K$, 
et soit $\mu$ la multiplicité générique de $X_{\mathrm{red},K}$ le long de
$Z_K$. On a alors $\ell_K=\mu \ell$. 
\end{lemm}

\begin{proof}
Choisissons un point $y$ sur $Z_K^\gen\subset X_K^\gen$ dont on note
$x$ l'image sur $X$ ; par construction, $x\in Z$. 
L'anneau local $\mathscr 
O_{X_{K,\groth},y}$ est artinien car $y\in X_K^\gen$, et il est plat sur 
$\mathscr O_{X_\groth,x}$ 
(\ref{ss-XKG-plat}); ce dernier est donc également artinien, ce qui montre
que $x$ appartient à $X^\gen$, et donc à $Z^\gen$ car $x\in Z$. 

Pour alléger les notations posons $A=\mathscr O_{X_\groth,x}$, 
$B= \mathscr 
O_{X_{K,\groth},y}$ et $M=F_{X_\groth,x}$. On note
$\mathfrak m$ l'idéal maximal de $A$ et $\kappa=A/\mathfrak m$ son
corps résiduel. Enfin on pose $B'=B/\mathfrak mB=\kappa\otimes_A B$. 
Remarquons que $\kappa=\mathscr O_{X_{\mathrm{red},\groth},x}$ 
et $B'=\mathscr O_{X_{\mathrm{red},K,\groth},y}$ 
(proposition \ref{prop-gfini} \emph d) et \emph c)). 

Par définition de $\ell$ il existe une filtration 
\[\{0\}=M_0\subset M_1\subset \ldots\subset M_\ell=M\]
du $A$-module $M$ dont les quotients successifs sont
$A$-isomorphes à $\kappa$. Comme $B$
est plat sur $A$, cette filtration en induit une
\[\{0\}=B\otimes_AM_0\subset B\otimes_A M_1\subset \ldots\subset B\otimes_A M_\ell=B\otimes_A M\]
dans la catégorie des $B$-modules dont les quotients successifs sont tous 
$B$-isomorphes à $B'$. Il vient
\[
\ell_K=\ell_B(B\otimes_A M)=\ell\cdot \ell_B(B')=\ell\cdot \ell_{B'}(B')=\ell \mu.\qedhere\]
\end{proof}

\section{Degré et degré calibré d'un morphisme quasi-fini}

\subsection{}\label{ss.degre}
Soit $f\colon Y\to X$ un morphisme d'espaces $k$-analytiques.
Soit $y\in Y$. 
Si $\dim_y(f^{-1}(f(y))>0$, on pose $\deg^y(f)=+\infty$.
Sinon, le point~$y$ est isolé dans sa fibre $f^{-1}(f(y))$,
l'anneau local de $f^{-1}(f(y))$ en~$y$ est une $\hr {f(y)}$-algèbre 
de dimension finie et on note
$\deg^y(f)$ sa dimension.

Lorsque $f$ est fini et plat en $y$, cet entier $\deg^y(f)$ coïncide avec l'entier $\deg^y(f)$
défini en \ref{ss-degfiniplat-local}. 

Nous dirons que $\deg^y(f)$ est le degré de~$f$ en~$y$.

Soit $V$ un domaine analytique de~$Y$ contenant~$y$
et soit $U$ un domaine analytique de~$X$ tel que $f(V)\subset U$;
soit $g\colon V\to U$ le morphisme déduit de~$f$.
Pour tout $y\in V$, on a $\deg^y(f)=\deg^y(g)$, 
car pour tout $y\in V$ d'image~$x\in X$, les anneaux locaux
en~$y$ de $f^{-1}(f(y))$ et de $g^{-1}(g(y))$ coïncident.

Soit $Z$ un sous-espace analytique fermé de~$X$ contenant~$x$
et soit $g\colon Y\times_X Z \to Z$ le morphisme déduit de~$f$
par changement de base. Comme les espaces $\hr x$-analytiques
$f^{-1}(x) $ et~$ g^{-1}(x)$ coïncident, on a $\deg^y(f)=\deg^y(g)$.

\subsection{} 
Soit $f\colon Y\to X$ un morphisme d'espaces $k$-analytiques.
Pour tout $x\in X$,
on définit  de le degré de~$f$ au-dessus de~$x$ comme la somme
\[ \deg_x(f) = \sum_{y\in f^{-1}(x)} \deg^y(f). \]
Il est fini si et seulement si $f^{-1}(x)$ est compact et de dimension nulle;
alors $f^{-1}(x)$ est un espace analytique fini sur~$\hr x$,
et $\deg_x(f)=\dim_{\hr x}(\mathscr O(f^{-1}(x)))$.

Lorsque $f$ est globalement fini et plat, cet entier $\deg_x(f)$ coïncide avec l'entier $\deg_x(f)$
défini en \ref{ss-degfiniplat-global}. 

\begin{lemm}\label{lemm.degre.K/k}
Soit $f\colon Y\to X$ un morphisme d'espaces analytiques.
Soit $K$ une extension valuée complète de~$k$; 
notons $p\colon X_K\to X$ et $q\colon Y_K\to Y$ 
les morphismes de changement de base,
et $f_K\colon Y_K\to X_K$ le morphisme déduit de~$f$ par changement de base.
Soit $x_K\in p^{-1}(x)$. 
\begin{enumerate}
\item Pour tout $y\in Y$ d'image~$x=f(y)\in X$,
tout point $x_K\in p^{-1}(x)$,
on a
\[   \deg^y(f) = \sum_{z\in f_K^{-1}(x_K) \cap q^{-1}(y)} \deg^z({f_K}). \]
\item Pour tout $x\in X$ et tout point~$x_K\in p^{-1}(x)$,
on a $\deg_{x}(f)=\deg_{x_K}(f_K)$.
\end{enumerate}
\end{lemm}
\begin{proof}
Tout se passe dans les fibres~$f^{-1}(x)$ et 
fibre $f_K^{-1}(x_K)=f^{-1}(x)\times \mathscr M(\hr{x_K})$.
Cela permet de supposer que $X=\mathscr M(k)=\{x\}$ et $X_K=\mathscr M(K)=\{x_K\}$.

Prouvons l'égalité~\emph a).
Si $\dim_y(Y)>0$, il en est de même de $\dim_{z}(Y_K)$
pour tout antécédent~$z$ de~$y$, si bien que les deux membres
de l'égalité sont infinis.
Supposons maintenant $\dim_y(Y)$. Alors $\mathscr O_{Y,y}$
est une $k$-algèbre artinienne et 
$\mathscr O_{Y,y}\otimes_k K=\prod_{z} \mathscr O_{Y_K,z}$,
où $z$ parcourt l'ensemble des antécédents de~$y$ dans~$Y_K$.
L'égalité s'ensuit.

L'égalité~\emph b) s'en déduit en sommant sur tous les points de~$Y$.
\end{proof}

\subsection{}\label{ss-sorites-longueurs}
Soit $A$ un anneau local artinien et soit $M$ un $A$-module de longueur finie.
Nous noterons $\delta_A(M)$ 
le nombre rationnel $\ell_A(M)/\ell_A(A)$. 
Indiquons quelques propriétés élémentaires de cet invariant. 

Si $A$ est réduit, c'est un corps, on a  $\ell_A(A)=1$ et $\delta_A(M)$ 
est donc simplement $\ell_A(M)$, c'est-à-dire 
la dimension de $M$ comme $A$-espace vectoriel. 

Supposons $M$ plat sur $A$, c'est-à-dire libre comme $A$-module ;
soit $r$ son rang. Comme $M\simeq A^r$,
$\ell_A(M)=r\ell_A(A)$, si bien que $\delta_A(M)=r$. 

Si $M=A_{\mathrm{red}}$, alors $\ell_A(M)=1$ et $\delta_A(M)=1/\ell_A(A)$. 

Soit $B$ une $A$-algèbre finie.
Soit $\kappa$ le corps résiduel de $A$ et $\lambda$ celui de $B$. 
Choisissons une filtration complète de~$B$ comme $B$-module
ayant $\ell_B(B)$ quotients successifs tous isomorphes à 
$\lambda$. On a alors $\ell_A(B)=\ell_B(B)\ell_A(\lambda)$, et $\ell_A(\lambda)$
est simplement la dimension de $\lambda$ comme $\kappa$-espace vectoriel, 
si bien qu'on peut finalement écrire
$\delta_A(B)=(\ell_B(B)/\ell_A(A))\dim_\kappa(\lambda)$. 

Soit $C$ une $B$-algèbre locale finie, de corps résiduel $\mu$. 
On a 
\begin{align*}
\delta_A(C)&=[\ell_C(C)/\ell_A(A)]\dim_\kappa \mu \\
&=[\ell_C(C)/\ell_B(B)][\ell_B(B)/\ell_A(A)]\dim_\lambda(\mu)\dim_\kappa(\lambda)\\
&=[\ell_C(C)/\ell_B(B)] \dim_\lambda(\mu)\cdot[\ell_B(B)/\ell_A(A)] \dim_\kappa(\lambda)\\
&=\delta_B(C)\delta_A(B).
\end{align*}

\subsection{}
Soit $f\colon Y\to X$ un morphisme entre espaces $k$-analytiques. Soit $y$ un point de $Y$ et soit $x$ son image sur $X$. 
Supposons que $f$ est quasi-fini en $y$  et 
que $x\in X^\gen$
(\ref{ss-sorites-xgen}), 
ce qui revient à demander que 
$\mathscr O_{X_\groth, x}$ soit artinien.
L'anneau local $\mathscr O_{Y_\groth,y}$ est alors artinien
car il est fini sur $\mathscr O_{X_\groth, x}$
(proposition \ref{prop-gfini}), 
si bien que $y\in Y^\gen$.
D'après la proposition~\ref{prop-gfini}, on a
$\dim_{\mathrm c}(Y_\groth,y)=\dim_{\mathrm c}(X_\groth, x)$ ; 
comme ces deux entiers sont respectivement égaux à $\dim_y Y$ et $\dim_x X$
puisque $x\in X^\gen$ et $y\in Y^\gen$, il vient 
$\dim_y Y=\dim_x X$.

Nous appellerons \textit{degré calibré} de~$f$ en~$y$ le nombre rationnel 
$\delta_{\mathscr O_{X_\groth, x}}(\mathscr O_{Y_\groth, y})$;
nous le noterons $\degcal^y(f)$. 

Soit $U$ un domaine analytique de~$X$ contenant~$x$
et soit $V$ un domaine analytique de~$Y$ contenant~$y$
tels que $f(V)\subset U$ ; soit $g\colon V\to U$
le morphisme déduit de~$f$;  on a $\degcal^y(f)=\degcal^y(g)$.
En effect, les anneaux G-locaux de~$U$ et~$X$ en~$x$ coïncident,
et les anneaux G-locaux de~$V$ et~$Y$ en~$y$ coïncident.

Soit $g\colon Z\to Y$ un morphisme et soit $z$ un antécédent de $y$
tel que $\dim_z g=0$. 
On déduit de \ref{ss-sorites-longueurs}
que $\degcal^z(f\circ g)=\degcal^y (f)\degcal^z(g)$. 

Soit~$V$ un domaine analytique de~$X$ contenant~$x$ 
et soit $j$ l'immersion de $V$ dans $X$. 
Comme $\mathscr O_{V_\groth, x}=\mathscr O_{X_\groth,x}$,
on a $\degcal^x(j)=1$. 

\subsection{}\label{def.degcal_x}
Soit $f\colon Y\to X$ un morphisme entre espaces $k$-analytiques. 
Soit $x$ un point de~$X^\gen$. 
Supposons que la fibre $f^{-1}(x)$ soit finie sur~$\hr x$.
On pose alors
\[ \degcal_x(f) = \sum_{y\in f^{-1}(x)} \degcal^y(f). \]

Si $V$ est un domaine analytique de~$X$ contenant~$x$
et $g\colon Y\times_X V\to V$
est le morphisme déduit par changement de base, on a
$ \degcal_x(g)=\degcal_x(f)$.

Soit $f\colon X_\red\to X$ l'immersion canonique.
Soit $x\in X^\gen$.
Par définition, on a 
\[ \degcal_x(f)=\degcal^x(f)=\delta_{\mathscr O_{X_\groth,x}}(\mathscr O_{X_{\red,\groth},x}) . \]
Comme $\mathscr O_{X_{\red,\groth},x}=(\mathscr O_{X_{\groth,x}})_\red$,
il en résulte que
\[ \degcal_x(f)= \delta_{\mathscr O_{X_\groth,x}}((\mathscr O_{X_{\groth},x})_\red)
 = 1/ \ell (\mathscr O_{X_\groth,x}) . \]
Autrement dit,
le degré calibré de~$f$ en~$x$ est l'inverse de la multiplicité
générique de~$X$ le long de l'unique composante
irréductible de~$X$ qui contient~$x$.

\begin{prop}\label{exemple-degcal-fini}
Soit $f\colon Y\to X$ un morphisme fini d'espaces $k$-analytiques 
et soit~$Z$ une composante irréductible de~$X$. 
Soit~$\lambda$ la longueur générique de $f_*\mathscr O_Y$ le long de~$Z$, 
et soit~$\mu$ la multiplicité générique de~$X$ le long de~$Z$. 
Pour tout $x\in Z^\gen$, on a $\degcal_x(f)=\lambda/\mu$.

En particulier, si $f$ est l'immersion canonique de~$X_\red$
dans~$X$, on a $\degcal_x(f)=1/\mu$.
\end{prop}
\begin{proof}
Soit $x\in Z^\gen\subset X^\gen$ et 
soit $y_1,\ldots, y_r$ les antécédents de $x$ sur $Y$.
Soit $V$ un bon domaine analytique de $X$ contenant $x$ ; 
posons $W=f^{-1}(V)$.  On a alors 
\begin{align*}
\degcal_x(f)&=\sum_i \degcal^{y_i}(f)\\
&=\frac 1{\ell(\mathscr O_{X_\groth,x})}\sum_i \ell_{\mathscr O_{X_\groth,x}}(\mathscr O_{Y_\groth,y_i})\\
&=\frac 1\mu
\sum_i \ell_{\mathscr O_{X_\groth,x}}(\mathscr O_{Y_\groth,y_i})\\
&=\frac 1\mu\ell_{\mathscr O_{X_\groth,x}}(\prod_i \mathscr O_{Y_\groth,y_i})\\
&=\frac \lambda \mu,
\end{align*}
où la dernière égalité résulte du corollaire \ref{coro-gfini}
et du fait que $f$ est fini, d'où la proposition. 
\end{proof}

\begin{lemm}\label{lemme-degcal-deg}
Soit $f\colon Y\to X$ un morphisme entre espaces $k$-analytiques. 
Soit $y$ un point de $Y$ et soit $x$ son image sur $X$. 
Supposons que $x\in X^\gen$
et que $f$ est quasi-fini et plat en $y$. 
On a alors l'égalité $\degcal^y(f)=\deg^y(f)$. 
\end{lemm}

\begin{proof}
Pour établir le lemme,
on peut remplacer $X$ par n'importe quel domaine analytique $V$ de $X$ contenant $x$ et $Y$
par n'importe quel domaine analytique $W$ de $f^{-1}(V)$ contenant $y$. 
Cela permet, au vu
du corollaire 
\ref{coro-quasifini-cofinal}, de supposer que $X$ et $Y$
ont affinoïdes et que $f$ est fini. Comme $f$ est plat en $y$, l'anneau local 
$\mathscr O_{Y,y}$ est un $\mathscr O_{X,x}$-module libre de rang $\deg^y(f)$. 
En vertu de la proposition~\ref{prop-gfini}, \emph c),
$\mathscr O_{Y_\groth,y}$ est 
un $\mathscr O_{X_\groth, x}$-module libre de rang $\deg^y(f)$, si bien que 
$\degcal^y(f)=\delta_{\mathscr O_{X_\groth, x}}(\mathscr O_{Y_\groth,y})=\deg^y(f)$. 
\end{proof}

\begin{coro}\label{coro-degcal-deg}
Soit $f\colon Y\to X$ un morphisme entre espaces $k$-analytiques. 
Soit~$y$ un point de~$Y$ et soit $x$ son image sur $X$. 
Supposons que $x\in X^\gen$, que
$f$ est quasi-fini en $y$ et que $X$ est réduit en~$x$. 
On a alors l'égalité $\degcal^y(f)=\deg^y(f)$. 
\end{coro}

\begin{proof}
Comme $X$ est réduit en~$x$,
l'anneau local artinien $\mathscr O_{X_\groth, x}$ est un corps. 
Par conséquent, l'anneau~$\mathscr O_{Y_\groth,y}$
est plat sur $\mathscr O_{X_\groth, x}$, si bien que $f$ est plat en~$y$ 
(\cite{ducros2024}, théorème~7.5) ; 
on conclut à l'aide du lemme précédent.
\end{proof}

\begin{prop}\label{prop.degre-multiplicatif}
Soit $g\colon Z\to Y$ et $f\colon Y\to X$ deux morphismes d'espaces $k$-analytiques. Soit $z$ un point de $Z$ et soit $y$ son image sur $Y$. 
On suppose que $f$ est quasi-fini en $y$ et que $g$
est quasi-fini en $z$. 

\begin{enumerate}
\item
Soit $X'$ un sous-espace analytique fermé
réduit d'un domaine analytique de $X$ 
tel que $x\in X'$ et que
$\dim_x X'=\dim_{\mathrm c}(X_\groth, x)$, et soit 
$g'$ le morphisme 
de $Z\times_X X'$ vers 
$Y\times_X X'$ déduit de $g$. 
On a alors $\deg^z(f\circ g)=\deg^y(f)\degcal^z(g').$

\item
Supposons que le morphisme~$g$ est plat en~$z$ ou que 
la fibre~$f^{-1}(x)$ est réduite en~$y$. 
Alors,
$\deg^z(f\circ g)=\deg^y(f)\deg^z(g)$.
\end{enumerate}
\end{prop}

Remarquons que $\dim_{\mathrm c}(X'_\groth, x)=\dim_x X'$
par construction, c'est-à-dire que $x\in (X')^\gen$, si bien que $\degcal^z(g')$ a bien un sens. 
Notons aussi que si $x\in X^\gen$ on peut prendre 
$X'=X_{\mathrm{red}}$. 

\begin{proof}
Démontrons~\emph a).
Pour alléger les notations, posons $Y'=Y\times_X X'$ et 
$Z'=Z\times_X X'$, et notons~$f'$ le morphisme de $Y'$ vers $X'$
induit par $f$. 

On a alors $\deg^z(f'\circ g')=\deg^z(f\circ g)$ et $\deg^y f'=\deg^y f$.
Comme $X'$ est réduit,
le corollaire \ref{coro-degcal-deg} entraîne
les égalités $\degcal^z(f'\circ g')=\deg^z(f'\circ g')$ et
$\degcal^y(f')=\deg^y(f')$, car $f'\circ g'$ et~$f'$ se
déduisent de~$f\circ g$ et~$f$ par changement de base.  
On a par ailleurs $\degcal^z(f'\circ g')=\degcal^y (f')\degcal^z (g')$;
il vient  donc
$\deg^z(f\circ g)=\deg^y(f)\degcal^z(g')$. 

Comme~$g'$ se déduit de~$g$ par changement de base,
on  a $\deg^z(g)=\deg^z(g')$.
Pour démontrer~\emph b), il suffit de donc vérifier que si $g$ est plat en~$z$ 
ou si $f^{-1}(x)$ est réduit en~$y$, alors 
$\degcal^z(g')=\deg^z(g')$. 

Si $g$ est plat en $z$, alors $g'$ est encore plat en $z$, si bien que
$\degcal^z(g')$ est égal à $\deg^z(g')$ (lemme \ref{lemme-degcal-deg}).

Supposons maintenant que $f^{-1}(x)$ est réduit en $y$. 
D'après le corollaire \ref{coro-degcal-deg}, 
il suffit de prouver que $Y'$ est réduit en~$y$. 
Pour ce faire, on peut remplacer~$X$ par n'importe quel domaine
analytique~$U$ de~$X$ contenant~$x$, l'espace~$X'$ par $U\times_X X'$ 
et l'espace~$Y$ par n'importe quel domaine analytique~$V$
de $Y\times_X U$ contenant~$y$. 
On peut donc tout d'abord supposer 
que $X'$ est un sous-espace analytique fermé de $X$ puis, 
en vertu du corollaire \ref{coro-quasifini-cofinal}
supposer que $X$ et $Y$ sont affinoïdes et que $f$ est fini. 
L'anneau local $\mathscr O_{X'_\groth, x}$ est artinien et réduit ; c'est donc un corps. Comme il est 
fidèlement plat sur l'anneau local  $\mathscr O_{X', x}$, ce dernier est également un corps. 
Le morphisme $f'$ étant sans bord en $y$, il résulte alors du théorème~6.3.3 de \cite{ducros2018}
que $\mathscr O_{f^{-1}(x), y}$ est fidèlement plat sur $\mathscr O_{Y',y}$. Par conséquent, l'anneau $\mathscr O_{Y',y}$ s'injecte dans un anneau réduit et est dès lors lui-même réduit, ce qui signifie que $Y'$ est réduit en $y$. 
\end{proof}

\section{Semi-continuité du degré}

\begin{prop}\label{prop.deg-y-scs}
Soit $f\colon Y\to X$ un morphisme d'espaces $k$-analytiques.
La fonction $y\mapsto \deg^y(f)$ de $Y\to\N\cup\{+\infty\}$ 
est semi-continue supérieurement.
\end{prop}
\begin{proof}
L'ensemble des points~$y\in Y$ en lesquels $\dim_y(f^{-1}(f(y))>0$
est un fermé de Zariski de~$Y$; c'est aussi l'ensemble des points
$y$ tels que $\deg^y(f)$ soit infini. Notons~$V$ l'ouvert complémentaire.
Il suffit de prouver que $y\mapsto \deg^y(f)$ 
est semi-supérieurement sur~$V$. Autrement dit, on peut supposer 
que $f$ est partout de dimension relative nulle 

L'assertion à montrer est G-locale sur $Y$ ;  on peut ainsi supposer que $X$ et $Y$
sont affinoïdes.
Soit $y\in Y$. 
D'après le \emph{Main theorem} analytique (\cite{ducros2007}, Thm. 3.2)
il existe une factorisation 
$V\xrightarrow g W \xrightarrow j T\xrightarrow h X$  de~$f$,
où $V$ est un voisinage affinoïde de $y$ dans $Y$, où
$g\colon V\to W$ est fini, $j\colon W\to T$ est l'immersion
d'un domaine affinoïde, et $h\colon T\to X$ est étale.

Le morphisme~$h$ étant fini et plat en tout point, la fonction~$t\mapsto \deg^t(h)$
est semi-continue supérieurement
(lemme \ref{lemm-finiplat-scs}). 
Il en est donc de même 
de sa restriction à $W$, qui n'est autre
que $w\mapsto \deg^w ({h \circ j})$.
La fonction $v\mapsto \deg^v(g)$ 
est semi-continue supérieurement.
Soit en effet $v\in V$ ; posons $d=\deg^v(g)$. Comme $g$ est finie, il existe
des voisinages affinoïdes~$\Omega$ de~$v$ dans $V$ et~$\Upsilon$ de~$g(v)$
dans $W$ tels
que $g$ induise un morphisme fini de~$\Omega$ sur~$\Upsilon$.
On peut également supposer que $v$ est le seul antécédent de~$g(v)$
dans~$\Omega$ ; le rang ponctuel du faisceau cohérent $g_*\mathscr O_{\Omega}$
en~$g(v)$ est donc égal à~$d$; par le lemme de Nakayama, il est majoré par 
$d$ sur un voisinage $\Upsilon'$ de $g(v)$ dans $\Upsilon$. 
Soit $u$ un point de $g^{-1}(\Upsilon')$. 
On a 
\begin{align*}
d&\geq \dim_{\hr {g(u)}}(g_*\mathscr O_{\Omega}\otimes \hr {g(u)})\\
&=\dim_{\hr {g(u)}}
\prod_{\omega\in \Omega\cap g^{-1}(g(u)}\mathscr O_{g^{-1}(g(u)),\omega}\\
&=\sum_{\omega \in \Omega \cap g^{-1}(g(u))}\deg^\omega(g),
\end{align*}
si bien que $\deg^\omega(g)\leq d$ pour tout $\omega$
appartenant à  $\Omega \cap g^{-1}(g(u))$ ; c'est en particulier le cas
pour $\omega=u$, ce qui achève de prouver la semi-continuité supérieure de $g$.

Comme un morphisme quasi-étale est quasi-fini et à fibres réduites,
la proposition~\ref{prop.degre-multiplicatif} entraîne 
la relation $\deg^v(f)=\deg^{g(v)}({h\circ j}) \deg^v(g)$
pour tout $v\in V$. 
Il s'ensuit alors que $v\mapsto \deg^v(g)$ est semi-continue supérieurement sur $V$, 
donc au voisinage de $y$, ce qui achève la démonstration. 
\end{proof}

\begin{prop}\label{prop.deg-x-scs}
Soit $f\colon Y\to X$ un morphisme d'espaces $k$-analytiques
qui est compact et séparé. 
La fonction $x\mapsto \deg_x(f)$ de~$X$ dans~$\N\cup\{+\infty\}$
est semi-continue supérieurement.
\end{prop}
\begin{proof}
Nous allons reprendre les arguments de la démonstration
de la proposition précédente.
Soit $Z$ le lieu des $y\in Y$ tels que $\dim(f^{-1}(f(y)))>0$ ;
c'est un fermé de Zariski de~$Y$. Par compacité de~$f$,
son image est une partie fermée de~$X$; c'est lieu des points~$x$
en lesquels $\deg^x(f)$ est infini.
Il suffit donc de démontrer que la fonction~$x\mapsto\deg^x(f)$
est semi-continue supérieurement sur son complémentaire.
Cela permet de se ramener au cas où $f$ est compacte
à fibres de dimension nulle. On peut raisonner G-localement sur $X$, ce qui permet de le supposer séparé ; 
dans ce cas, $Y$ est également séparé. 

Soit $x$ un point de~$X$.
La fibre $f^{-1}(x)$ est finie; 
en appliquant 
le \emph{Main theorem} de~\cite{ducros2007}
à chacun de ses éléments (ce qui est possible
car $Y$ est séparé) et en utilisant la compacité de~$f$,
on obtient l'existence d'un diagramme
$V\xrightarrow g W \xrightarrow j T\xrightarrow h X$  de~$f$,
où $V$ est un voisinage analytique compact de~$f^{-1}(x)$,
$W$ est un espace analytique compact,
$g\colon V\to W$ est fini, $j\colon W\to T$ est l'immersion
d'un domaine analytique, et $h\colon T\to X$ est étale.
Par compacité de~$f$, l'image $f(V)$ est un voisinage de~$x$.
Il suffit ainsi de démontrer que $\deg^\cdot (h\circ j\circ g)$
est semi-continue supérieurement.

La fonction $w\mapsto \deg_w(g)\colon W \to \N$ est égale à 
$w\mapsto \dim_{\hr w}(g_*\mathscr O_V)\otimes \hr w$ ; elle est semi-continue
supérieurement par le lemme de Nakayama.

Soit $(w_i)$ la famille des antécédents de~$x$ dans~$W$;
pour tout~$i$, posons $d_i=\deg^{w_i}(h\circ j)=\deg^{w_i}(h)$.
Par définition, on a $\deg_x(h\circ j)=\sum d_i$.
Comme $h$ est étale, c'est un morphisme fini et plat en tout point;
par suite, pour tout~$i$, il existe un voisinage analytique compact~$T_i$
de~$w_i$ dans~$T$ qui est fini et plat sur un voisinage
analytique compact~$X_i$ de~$x$ dans~$X$.
On peut supposer que les~$T_i$ sont
deux à deux disjoints, que $T_i$ est de degré constant $d_i$ sur $X_i$,
et que $\deg_w(g)\leq _{w_j}(g)$ pour tout $w\in j^{-1}(T_i)$. 

La réunion des $j^{-1}(T_i)$ est un voisinage de~$(h\circ j)^{-1}(x)$ 
dans~$W$; comme $W$ est compact, il existe un voisinage~$X'$
de~$x$ dans~$\bigcap X_i$  tel que $(h\circ j)^{-1}(X')$ soit contenu
dans la réunion de $j^{-1}(T_i)$.

Pour $x'\in X'$, le degré $\deg_{x'}(h\circ j)$
est la somme des $\deg_{x'}(h\circ j)|_{j^{-1}(T_i)}$,
qui est majoré par $\sum d_i$.

Soit $x'$ un point de~$X'$; par définition,
\[ \deg_{x'}(f)= \sum_{y'\in f^{-1}(x')} \deg^{y'} (f)
= \sum_i \sum_{\substack{y'\in f^{-1}(x') \\ g(y') \in j^{-1}(T_i)}} \deg^{y'}(f).\]
Comme $h\circ j$ est quasi-étale, la proposition~\ref{prop.degre-multiplicatif}
entraîne
l'égalité $\deg^{y'}(f)=\deg^{g(y')}(h\circ\nobreak j)\deg^{y'}(g)$. Par suite,
\begin{align*}
 \deg_{x'}(f) 
& = \sum_i \sum_{\substack{y'\in f^{-1}(x') \\ g(y') \in j^{-1}(T_i)}} \deg^{g(y')}(h\circ j) \deg^{y'}(g) \\
& = \sum_i \sum_{\substack{w' \in j^{-1}(T_i) \\ h\circ j(w')=x'}}
 \sum_{y'\in g^{-1}(w')} \deg^{g(y')}(h\circ j) \deg^{y'}(g) \\
&= \sum_i \sum_{\substack{w' \in j^{-1}(T_i) \\  h\circ j(w')=x'}}
 \deg^{w'} (h\circ j)  \deg_{w'}(g). 
\end{align*}
Appliqué à $x'=x$, cette formule démontre que 
\[ \deg_{x}(f) = \sum_i d_i \deg_{w_i(g)}. \]
On a par ailleurs
par choix des $T_i$ les inégalités
$\deg_{w'}(g)\leq \deg_{w_i}(g)$,
et $\deg^{w'}(h\circ j)\leq d_i$
pour tout $i$ et tout $w'\in j^{-1}(T_i)$ ; il vient alors 
\[ \deg_{x'}(f) \leq \sum_i d_i \deg_{w_i}(g)  = \deg_x(f) \]
pour tout $x'$ dans~$X'$.
Cela entraîne la semi-continuité voulue.
\end{proof}

\begin{rema}
L'hypothèse de séparation est indispensable à la validité de la proposition 
\ref{prop.deg-x-scs} ci-dessus. En effet, soit $X$ le disque unité fermé et
soit $X'$ son domaine affinoïde défini par l'égalité $\abs T=1$. 
Soit $Y$ l'espace $k$-analytique
compact obtenu en recollant deux copies de $X$ le long de $X'$, 
le recollement se faisant via l'identité de $X'$ (et non via l'inversion, qui conduirait
à la droite projective) ; on pourrait alternativement
décrire $Y$ comme la fibre générique de la droite affine formelle
avec origine dédoublée. 
Soit $f \colon Y\to X$ le morphisme induit par l'identité sur chacune
des deux copies de $X$ constituant $Y$. Alors $\deg_x(f)=1$ si $x\in X'$ et $\deg_x(f)=2$ si $x\notin X'$. 
Par conséquent $x\mapsto \deg_x(f)$ est semi-continue inférieurement, mais pas supérieurement. 
\end{rema}

\begin{lemm}
Soit $f\colon Y\to X$ un morphisme d'espaces $k$-analytiques.
Soit $x$ un point de~$X$ et soit $W$ un domaine
analytique compact de~$f^{-1}(x)$.
Lorsque $V$ et $V'$ sont des voisinages analytiques 
compacts de~$W$ dans~$Y$ tels
 que $V \cap f^{-1}(x)=V'\cap f^{-1}(x)=W$, 
les fonctions $\xi\mapsto \deg_\xi(f|_V)$ et $\xi\mapsto \deg_\xi(f|_{V'})$
coïncident au voisinage de~$x$.
\end{lemm}
\begin{proof}
Par construction, $V\cap V'$ est un voisinage de~$V\cap f^{-1}(x)$ dans~$V$.
Comme $V$ est compact, l'application~$f|_V$ est propre;
il existe donc un voisinage ouvert~$U$ de~$x$ tel que $V\cap f^{-1}(U)\subset V\cap V'$, c'est-à-dire $f^{-1}(U)\cap V\subset f^{-1}(U)\cap V'$.
Par symétrie, il existe un voisinage ouvert~$U'$ de~$x$
tel que $f^{-1}(U')\cap V'\subset f^{-1}(U')\cap V$.
L'ouvert $U''=U\cap U'$ est un voisinage de~$x$
tel que $f^{-1}(U'')\cap V=f^{-1}(U'')\cap V'$.
En revenant aux définitions des fonctions $\deg_\xi(f|_V)$ et
$\deg_\xi(f|_{V'})$, on en déduit qu'elles coïncident sur~$U''$.
\end{proof}

\subsection{}\label{def-degre-bizarre}
Soit $f\colon Y\to X$ un morphisme d'espaces $k$-analytiques.
Soit $y$ un point de~$Y$ et soit $x$
son image sur $X$. Supposons que $y$ est isolé dans sa fibre~$f^{-1}(x)$.
Alors, $\{y\}$ est un domaine compact de~$f^{-1}(x)$
et il existe un voisinage analytique 
compact~$V$ de~$y$ dans~$Y$
tels que $V\cap f^{-1}(x)= \{y\}$.
D'après le lemme précédent, le germes en~$x$
de l'application $\xi\mapsto \deg_\xi(f|_V)$ ne dépendent
pas du choix de~$V$; on note $\xi\mapsto \deg^{(y)}_\xi(f)$ ce germe
d'application à valeurs dans $\N\cup\{+\infty\}$.

D'après la proposition~\ref{prop.deg-x-scs}, 
si le germe $(Y,y)$ est séparé, 
le germe d'application $\xi\mapsto \deg^{(y)}_\xi(f)$
est semi-continu supérieurement en~$x$.

Si $f$ est fini et plat en~$y$, le germe $\xi\mapsto \deg^{(y)}_\xi(f)$
est localement constant en~$x$.
En effet, il existe un voisinage analytique compact~$V$ de~$y$
tel que $f^{-1}(x)\cap V=\{y\}$ 
et un voisinage analytique compact~$U$ de~$x$
tels que $f$ induise un morphisme fini et plat de~$V$ sur~$U$.
L'assertion découle alors du~\S\ref{ss.deg-x-fini-plat}.

\section{Cycles}
\index{cycle}

Soit $X$ un espace $k$-analytique.

\begin{defi}
On appelle \emph{cycle} sur~$X$ une application~$z$ 
de l'ensemble des fermés de Zariski intègres (i.e., irréductibles et réduits)
de~$X$
dans~$\Z$ qui est \emph{localement finie} au sens suivant:
pour tout domaine affinoïde~$V$ de~$X$, l'ensemble des fermés
de Zariski intègres~$Y$ de~$X$ tels que $Y\cap V\neq\emptyset$ et $z(Y)\neq 0$ 
est fini.
\end{defi}

On notera  $\sum z(Y) [Y]$ le cycle $z$; cette notation est abusive
lorsque l'ensemble $E$ des fermés de Zariski~$Y$ de~$X$ tels que $z(Y)\neq 0$
est infini. 
On dit que $z$ est \emph{effectif} si l'on a $z(Y)\geq 0$ pour tout~$Y$.
Comme $E$ est G-localement fini, la réunion des $Y$ pour $Y$ parcourant $E$
est un fermé de Zariski de $X$, appelé le \emph{support}
du cycle $z$ ; on le note $\abs z$. 
Observons que $\abs z=\emptyset$ si et seulement si $z=0$.

L'ensemble $\gZ(X)$ est un groupe abélien. 

On dira qu'un cycle~$z$ est purement $p$-dimensionnel (resp. purement
$p$-codimensionnel) si $z(Y)\neq 0$ entraîne $\dim(Y)=p$
(resp. $\codim(Y)=p$). Dans ce cas, son support est purement
$p$-dimensionnel (resp. purement $p$-codimensionnel).
On notera $\gZ_p(X)$ (resp. $\gZ^p(X)$) le sous-groupe de~$\gZ(X)$
formé des cycles purement $p$-dimensionnels 
(resp. purement $p$-codimensionnels).

\subsection{}
Soit $Y$ un sous-espace analytique fermé de~$X$.
Soit $(Y_i)$ la famille des composantes irréductibles de~$Y$; 
pour tout~$i$, notons $n_i$ la multiplicité
générique de~$Y$ le long de~$Y_i$
(\ref{defi-longueur-generique}). 
Alors,  $[Y]=\sum n_i [Y_{i,\red}]$ est un  cycle sur~$X$,
qu'on appelle le cycle associé à~$Y$; il est effectif.

\begin{prop}\label{prop.cycle-restriction}
\begin{enumerate}
\item
Soit $V$  un domaine analytique de~$X$.
Il existe un unique morphisme de groupes $\gZ(X)\to \gZ(V)$,
noté $z\mapsto z|_V$,
qui commute aux sommes localement finies et
qui,  pour tout sous-espace analytique fermé~$Y$ de~$X$,
applique le cycle~$[Y]\in \gZ(X)$ sur le cycle  $[Y\cap V]\in \gZ(V)$.

\item
Soit $V$ un domaine analytique de~$X$. Pour tout $z\in \gZ(X)$,
on a $\abs{z|_V}=\abs z\cap V$. En particulier, 
pour que $z|_V=0$, il faut et il suffit que le support de~$z$
soit disjoint de~$V$.

\item
Ces applications font de  $V\mapsto \gZ(V)$ 
un G-préfaisceau en groupes abéliens sur~$X$;
c'est un faisceau.
\end{enumerate}
\end{prop}
\begin{proof}
Démontrons~\emph a). 
Soit $z$ un cycle sur~$X$; pour tout fermé irréductible~$Y$
de~$X$, soit~$n_Y=z(Y)$ la multiplicité de~$Y$ dans~$z$.
Soit $T$ un fermé de Zariski irréductible de~$V$ et soit $Y$
son adhérence de Zariski dans~$X$.
Si $\dim(Y)=\dim(T)$, on pose $n_T=n_Y$; sinon, on pose $n_T=0$.
La famille des~$Y$ tels que~$n_Y\neq 0$ est localement finie sur~$X$,
par définition d'un cycle. 
La famille des~$T$
tels que~$n_T\neq 0$ est la famille des composantes
irréductibles de $Y\cap V$, où $Y$ est un fermé irréductible de~$X$
tel que $n_Y\neq0$. Elle est donc encore localement finie.
Ainsi, la famille des~$(n_T)_T$ est un cycle sur~$V$ ;
on le note~$z|_V$. L'application $z\mapsto z|_V$ est manifestement
un morphisme de groupes qui commute aux sommes localement finies.

Soit $Y$ un sous-espace analytique fermé de~$X$. 
Soit $(Y_i)$ la famille de ses composantes irréductibles
et soit~$n_i$ la multiplicité générique de~$Y$ le long de~$Y_i$.
Le cycle~$z=[Y]$ qui est associé à~$Y$ est
donné par $\sum n_i [Y_i]$.
Pour tout~$i$, soit $(Z_{ij})$ la famille des composantes irréductibles
de~$Y_i\cap V$. D'après la proposition~\ref{prop.long-restr},
appliquée à l'espace~$Y$, à son domaine~$Y\cap V$
et au faisceau~$\mathscr O_Y$,
la multiplicité de~$Y\cap V$ le long de~$Z_{ij}$
est égale à~$n_i$. Par conséquent, le cycle associé à~$Y\cap V$
est égal à~$\sum_{i,j} n_i Z_{ij}$;
c'est également le cycle $[Y]|_V$ défini ci-dessus.
Le morphisme $z\mapsto z|_V$ satisfait donc les conditions de
l'assertion~\emph a); vérifions que c'est le seul.

Soit $\rho\colon \gZ(X)\to \gZ(V)$ un morphisme satisfaisant
ces conditions; démontrons que $\rho(z)=z|_V$ pour tout cycle~$z$.
Par additivité localement finie, il suffit de le vérifier
lorsque $z=[Y]$, où $Y$ est un fermé de Zariski intègre de~$X$.
Soit $(Z_i)$ la famille des composantes irréductibles de $Y\cap V$.
Comme $Y$ est réduit, son domaine~$Y\cap V$ est réduit,
de sorte que $\rho(Y)=[Y\cap V] = \sum [Z_i]=[Y]|_V$.

\medskip

L'assertion~\emph b) découle de la construction de morphisme de
restriction des cycles.

Démontrons maintenant~\emph c).
Soit $V$ et $W$ des domaines analytiques de~$X$ tels que $W\subset V$
et soit $Y$ un sous-espace analytique fermé de~$X$.
D'après ce qui précède,  on a 
\[ [Y]|_{W} = [Y\cap W] = [(Y\cap V)\cap W]= [Y\cap V]|_W 
= ([Y]|_V )|_W. \]
Cela entraîne que l'application $V\mapsto \gZ(V)$ est 
un  G-préfaisceau pour ces applications de restriction.
Démontrons que ce préfaisceau est un faisceau.

Soit $V$ un domaine analytique de~$X$, soit $(V_i)$ un G-recouvrement de~$V$,
pour tout couple~$(i,j)$, posons $V_{ij}=V_i\cap V_j$.
Pour tout~$i$, soit $z_i\in \gZ(V_i)$; supposons que
pour tout couple~$(i,j)$, on ait $z_i|_{V_{ij}}=z_j|_{V_{ij}}$
et démontrons qu'il existe un unique cycle $z\in \gZ(V)$
tel que $z|_{V_i}=z_i$ pour tout~$i$.

L'unicité d'un tel cycle~$z$ découle de~\emph b): 
si $z|_{V_i}=0$ pour tout~$i$, alors le support de~$z$
ne rencontre aucun des~$V_i$, donc est vide, donc $z=0$.

Prouvons maintenant l'existence d'un cycle~$z$ tel que $z|_{V_i}=z_i$
pour tout~$i$.

Soit $Y$ un sous-espace analytique fermé intègre de~$X$ 
et soit $m$ un entier.
Pour tout~$i$, soit $T_i$ la réunion des composantes irréductibles
de $Y\cap V_i$ apparaissant avec multiplicité~$m$ dans le cycle~$z_i$;
c'est un fermé de Zariski de~$Y\cap V_i$.
Le sous-espace fermé $T_i\cap V_j$ de~$V_{ij}$ 
est la réunion des composantes irréductibles
de $T\cap V_{ij}$ qui apparaissent avec multiplicité~$m$ dans 
le cycle~$z_i|_{V_{ij}}$. Comme $z_i|_{V_{ij}}=z_j|_{V_{ij}}$,
on a $T_i\cap V_j=T_j\cap V_i\subset T_j$.
Soit $T^{(m)}$ 
la réunion des~$T_i$; d'après ce qui précède, on a $T^{(m)}\cap V_i=T_i$
pour tout~$i$. Par suite, $T^{(m)}$ est un fermé de Zariski de~$V$
contenu dans~$Y$.

La famille $(Y\cap V_i)$ est un G-recouvrement de~$Y$;
il existe donc un indice~$i$  et une composante irréductible~$E$ de $Y\cap V_i$;
soit $m$ sa multiplicité dans~$z_i$.
Comme $E$ est Zariski dense dans~$Y$ et comme $T^{(m)}$ contient~$E$,
on a $T^{(m)}=Y$. Par suite, $Y\cap V_j=T_j$
pour tout~$j$; autrement dit, pour tout~$j$,
toute composante irréductible de $Y\cap V_j$
est de mulitplicité~$m$ dans~$z_j$, si bien que 
$T^{(n)}=\emptyset $ si $n\neq m$. Posons $z(Y)=m$.

Démontrons maintenant que la famille $z=(z(Y))_Y$ est un cycle sur~$V$.
Fixons un indice~$i$ et considérons un domaine analytique compact~$W$
de~$V$ contenu dans~$V_i$.
Par définition d'un cycle, l'ensemble des sous-espaces fermés~$T$ de~$V_i$
tels que $z_i(T)\neq 0$ et $T$ rencontre~$W$ est fini.
Par construction de l'application $Y\mapsto z(Y)$,
pour tout sous-espace analytique fermé intègre~$Y$ de~$V$
qui rencontre~$W$ et tel que $z(Y)\neq 0$,
il existe une composante irréductible~$T$ de~$Y\cap V_i$
telle que $z_i(Y)=z_i(T)\neq 0$ et $T$ rencontre~$W$; alors $Y$
est l'adhérence de Zariski de~$T$ dans~$V$.
Par suite, l'ensemble des sous-espaces analytiques fermés intègres~$Y$ de~$V$
tels que $z(Y)\neq 0$ et $Y\cap W\neq\emptyset$ est fini.
Comme $(V_i)$ est un G-recouvrement de~$V$, cela entraîne
que la famille~$(z(Y))_Y$ est localement finie; ainsi, $z$ est un cycle
sur~$V$.

Démontrons maintenant que $z|_{V_i}=z_i$ pour tout~$i$.
Soit $T$ un sous-espace analytique fermé intègre de~$V_i$
et soit $Y$ son adhérence de Zariski dans~$V$.
Compte tenu de la proposition~\ref{prop.cycle-restriction}
définissant la restriction d'un cycle à un domaine,
il suffit de démontrer que l'on a
$z_i(T)=z(Y)$ si $\dim(Y)=\dim(T)$, et $z_i(T)=0$ sinon.

Supposons d'abord $\dim(Y)=\dim(T)$; dans ce cas, $T$ est
une composante irréductible de~$Y\cap V$, donc on a $z(Y)=z_i(T)$
par construction du cycle~$z$.

Supposons maintenant $\dim(Y)\neq \dim(T)$ et démontrons
que $z_i(T)=0$; on va raisonner par contraposition et,
supposant que $z_i(T)\neq 0$, prouver que $\dim(Y)=\dim(T)$.
Pour tout~$j$, notons~$Z'_j$ la réunion des composantes de
dimension~$\dim(T)$ du cycle~$z_j$ ; pour $j=i$,
le fermé~$T$ de~$V_i$ apparaît dans~$z_i$,
car $z_i(T)\neq0$, si bien que $Z'_i$ contient~$T$.

Soit $Z'$ la réunion des~$Z'_j$. 
Le sous-espace fermé $Z'_j\cap V_k$ de~$V_{jk}$ est la réunion des composantes
du cycle~$z_j|_{V_{jk}}$ qui sont de dimension~$\dim(T)$.
Comme $z_j|_{V_{jk}}=z_k|_{V_{jk}}$, on $Z'_j\cap V_k=Z'_k\cap V_j
\subset Z'_k$. On a donc $Z'\cap V_j=Z'_j$ pour tout~$j$,
ce qui entraîne que $Z'$ est un fermé de Zariski de~$V$;
il est purement de dimension~$\dim(T)$. 

Comme $Z'$ contient~$Z'_i$, il contient~$T$,
donc également son adhérence de Zariski~$Y$, si bien 
que $\dim(Y)\leq \dim(Z')=\dim(T)$. Ainsi, $\dim(Y)=\dim(T)$.
Cela termine la démonstration de la proposition.
\end{proof}

\section{Fonctions méromorphes en géométrie analytique}

\begin{lemm}\label{lemm.regulier}
Soit $A$ une algèbre affinoïde et soit $f\in A$. 
Les propriétés suivantes sont équivalentes:
\begin{enumerate}\def\theenumi{\arabic{enumi}}\def\labelenumi{(\theenumi)}
\item La fonction~$f$ est un élément régulier de~$A$;
\item Le germe~$f_x$ est régulier pour tout $x\in\Spec(A)$;
\item Le germe~$f_x$ est régulier pour tout $x\in\mathscr M(A)$;
\item Pour tout domaine affinoïde~$V$ de~$\mathscr M(A)$,
$f|_V$ est un élément régulier de~$\mathscr A_V$;
\item Il existe un recouvrement $(V_i)$ de~$\mathscr M(A)$
par des domaines affinoïdes tels que pour tout~$i$, $f|_{V_i}$
soit un élément régulier de $\mathscr A_{V_i}$.
\end{enumerate}
\end{lemm}
\begin{proof}
%
%
L'équivalence des propriétés~(1) et~(2) découle 
de la fidèle platitude du morphisme $ \coprod_{x\in\Spec(A)} \Spec(A_x)
\to\Spec(A)$,
celle de~(2) et~(3) de la fidèle platitude du morphisme  naturel
d'espaces localement annelés $\mathscr M(A)\ra\Spec(A)$.
De même, l'implication (3)$\Rightarrow$(4) découle de la platitude
du morphisme d'inclusion $V\ra\mathscr M(A)$.
L'implication (4)$\Rightarrow$(5) est évidente.
Enfin (5) entraîne~(3) car, sous ces hypothèses,
le morphisme $\coprod_i V_i\ra \mathscr M(A)$ est fidèlement plat.
\end{proof}

\subsection{}
Soit $X$ un espace analytique.
Soit $\mathscr S_X$ le sous-G-faisceau du G-faisceau~$\mathscr O_X$
dont les sections sur un domaine analytique~$V$
sont les fonctions holomorphes sur~$V$ dont la restriction
à tout domaine affinoïde~$U$ de~$V$ est un élément régulier 
dans~$\mathscr O(U)$.
Il découle du lemme~\ref{lemm.regulier} que pour
tout domaine affinoïde~$V$ de~$X$, $\mathscr S_X(V)$
est l'ensemble des éléments réguliers de~$\mathscr O_X(V)$.

\subsection{}
Par définition, le faisceau des
\emph{fonctions méromorphes}~$\mathscr M_X$\index{fonction méromorphe}
sur~$X$ est le faisceau pour la G-topologie associé au préfaisceau 
$U\mapsto \mathscr S_X(U)^{-1} \mathscr O_X(U)$.
Le morphisme naturel de préfaisceaux donné par $f\mapsto f/1$  
induit un morphisme de~$\mathscr O_X$ dans~$\mathscr M_X$;
il est injectif.

Lorsque $U$ est un domaine affinoïde de~$X$, on a encore
un morphisme naturel, injectif, de l'anneau total des fractions 
de~$\mathscr O_X(U)$ dans~$\mathscr M_X(U)$.

\subsection{}
Soit $f$ une section globale de~$\mathscr M_X$. 

On définit le \emph{faisceau des dénominateurs}~$\mathscr D_f$ de~$f$
comme le faisceau d'idéaux de~$\mathscr O_X$ constitué des sections~$a$
telles que~$af$ appartienne à l'image
de~$\mathscr O_X$ dans~$\mathscr M_X$.
C'est un faisceau cohérent. En effet, on peut se borner à vérifier
cette assertion sous l'hypothèse que $X$ est
un espace affinoïde sur lequel 
$f$ s'écrit $g/h$, pour $g\in\mathscr O_X(X)$ et $h\in\mathscr S_X(X)$.
Comme $h$ n'est pas diviseur de zéro, 
$\mathscr D_f$ apparaît comme le noyau du morphisme
$a\mapsto haf $ de~$\mathscr O_X$ dans~$\mathscr O_X/h\mathscr O_X$,
d'où l'assertion.

Comme $\mathscr D_f$ est un faisceau cohérent d'idéaux, son lieu
des zéros~$D_f$ est un fermé de Zariski de~$X$
(\cf\cite{ducros2009}, \S0.2.3).
Le complémentaire de~$D_f$ est le plus grand ouvert de Zariski de~$X$
sur lequel $f$ est holomorphe.
Par définition d'un élément régulier, 
si $V$ est un domaine affinoïde
de~$X$ sur lequel $f$ s'écrit $g/h$, avec $g\in\mathscr O_X(V)$
et $h\in\mathscr S_X(V)$, on a $D_f\cap V\subset \mathscr V(h)$,
donc $D_f\cap V$ ne contient aucune composante irréductible de~$V$.
D'après le th.~4.20 de~\cite{ducros2009},
pour tout domaine analytique~$V$ de~$X$, $V\cap D_f$ ne contient 
aucune composante irréductible de~$V$.

\begin{lemm}
Soit $A$ une algèbre affinoïde, soit $X=\mathscr M(A)$.
L'homomorphisme injectif de~$\Frac(A)$,
l'anneau total des fractions de~$A$, dans~$\mathscr M_X(X)$
est un isomorphisme.
\end{lemm}
\begin{proof} 
Démontrons qu'il est surjectif.
Soit $f\in\mathscr M_X(X)$, soit $\mathscr D_f$ son faisceau des dénominateurs,
soit~$I$ l'idéal correspondant de~$A$.

Démontrons que $I$ contient un élément régulier~$h$ de~$A$.

Raisonnons par l'absurde. Si $I$ ne contient aucun élément régulier,
c'est qu'il est contenu dans un idéal premier associé à~$A$,
car $A$ est un anneau noethérien (Bourbaki, AC IV, \S1, \no4, prop.~8).
Soit $\mathfrak p$ un tel idéal premier associé et soit $a$
un élément de~$A$ dont~$\mathfrak p$ soit l'annulateur.
Comme $\mathscr D_f$ est le faisceau cohérent associé à
l'idéal~$I$, on a, 
pour tout domaine analytique~$V$
de~$\mathscr M(A)$
l'égalité 
$\mathscr D_f(V)=I\mathscr O_X(V)$.
Par suite, pour tout tel domaine~$V$
 et toute section $h\in\mathscr D_f(V)$,
on a 
$a|_V h=0$.

Soit maintenant $(V_i)$ une famille finie de domaines affinoïdes de~$X$
recouvrant~$X$ telle que sur chaque~$V_i$, la fonction méromorphe~$f$ 
s'écrive $g_i/h_i$,
avec $g_i\in\mathscr O_X(V_i)$ et $h_i\in\mathscr S_X(V_i)$.
Pour tout~$i$, on a $h_i\in\mathscr D_f(V_i)$, donc $a|_{V_i}h_i=0$,
puis $a|_{V_i}=0$ puisque $h_i$ n'est pas diviseur de zéro
dans~$\mathscr O_X(V_i)$. Puisque les $V_i$ recouvrent~$X$,
il vient $a=0$, ce qui est absurde, son annulateur étant égal à~$\mathfrak p$.

Si $h$ est un élément régulier de~$A$ qui appartient à~$I$, 
la fonction méromorphe~$hf$ est une section globale de~$\mathscr O_X(X)$, 
donc appartient
à~$A$. Par conséquent, $f=(hf)/h$ appartient à~$\Frac(A)$.
\end{proof}

\begin{coro} Soit~$X$ un bon espace~$k$-analytique
et soit~$f$ une fonction méromorphe sur~$X$. 
Pour tout~$x\in X$ il existe un voisinage ouvert~$U$ de~$x$,
une fonction~$g\in \mathscr O_X(U)$ et une fonction~$h\in \mathscr S_X(U)$
telles que~$f|_U=g/h$.
\end{coro} 

\subsection{}\label{mero-reg}
Soit $X$ un espace analytique et soit $f\in\mathscr M_X(X)$
une fonction méromorphe sur~$X$.
On dit que $f$ est une 
fonction méromorphe régulière\index{fonction méromorphe!--- régulière}
(\cf \cite{ega4.4}, 20.1.8)
si pour tout
domaine affinoïde~$V$ de~$X$,
la restriction de~$f$ à~$V$ est un élément
régulier de
l'anneau~$\mathscr M_X(V)$. Alors, $f$ est un élément inversible
de l'anneau $\mathscr M_X(X)$.
De plus, il existe un ouvert de Zariski~$U$ de~$X$,
dont le complémentaire est rare, tel que $f$ définisse
une fonction holomorphe inversible sur~$U$.

Observons que le faisceau $\mathscr M_X^\times$ des
fonctions méromorphes régulières est le faisceau en groupes
abéliens associé au faisceau des fonctions holomorphes
régulières sur~$X$.

\section{Diviseur d'une fonction méromorphe régulière}

\subsection{}
Soit~$X$ un espace~$k$-analytique.
Soit~$f\in\mathscr S_X(X)$ une fonction analytique régulière sur~$X$.
Le sous-espace analytique fermé de~$X$ d'équation~$f=0$ 
est purement de codimension~$1$, il est vide si et seulement
si $f$ est inversible.
Notons~$\div(f)$ le cycle sur~$X$ associé; ce cycle est effectif
et nul si et seulement si $f$ est inversible.

Pour tout sous-espace analytique fermé intègre~$Y$ de~$X$,
notons $m_Y(f)$ le coefficient de~$Y$ dans~$\div(f)$.

\begin{lemm}\label{lemm.div(f)-local}
Soit $f$ une fonction analytique régulière sur~$X$.
Soit $Y$ un sous-espace analytique fermé intègre de~$X$.
Soit $V$ un bon domaine analytique de~$X$ et soit $y\in V\cap Y$
tel que l'anneau local $\mathscr O_{V\cap Y,y}$ soit artinien.
On a $m_Y(f) = \Long( \mathscr O_{V,y}/(f))$.
\end{lemm}
\begin{proof}
Si $f(y)\neq 0$, alors $Y$ n'est pas une composante irréductible
de $\{f=0\}$. Dans ce cas, $m_Y(f)=0$ ;
de plus, $f$ est inversible dans $\mathscr O_{V,y}$
donc $\Long(\mathscr O_{V,y}/(f))=0$.

Supposons maintenant $f(y)=0$.
Comme $\{y\}$ est Zariski dense dans~$Y$
(voir les commentaires après la définition~\ref{defi.anneau-G-local}),
on a $Y\subset\{f=0\}$.
Puisque $Y$ est purement de codimension~$1$, c'est une composante
irréductible de~$\{f=0\}$. Par définition,
$m_Y(f)$ est la longueur générique du faisceau cohérent  $\mathscr O_X/(f)$
le long de~$Y$. D'après la proposition~\ref{prop-longueur-generique},
c'est la longueur de l'anneau $\mathscr O_{V,y}/(f)$.
\end{proof}

\begin{lemm}\label{divadd}
Si~$f$ et~$g$ sont deux fonctions holomorphes régulières
sur~$X$, alors~$\div(fg)=\div (f)+\div (g)$.  
\end{lemm}
\begin{proof}
Soit $Y$ un sous-espace analytique fermé intègre de~$X$,
de codimension~$1$; démontrons que $m_Y(fg)=m_Y(f)+m_y(g)$.
Soit $V$ un domaine affinoïde de~$X$ et soit $y\in V\cap Y$
tel que $\mathscr O_{V\cap Y,y}$ soit artinien (il en existe).
Notons $\mathscr A=\mathscr O_{V,y}$;
la suite exacte 
\[0\to g\mathscr A/(fg)\to \mathscr A(fg)\to \mathscr A/(g)\to 0\] 
peut se récrire  \[0\to \mathscr A/(f)\to \mathscr A/fg\to \mathscr 
A/(g)\to 0\] puisque~$g$ n'est pas diviseur de zéro dans~$\mathscr A$. 
on a donc l'égalité 
$\Long(\mathscr O_{V,y}/(fg))=\Long(\mathscr O_{V,y}/(f))
+ \Long(\mathscr O_{V,y}/(g))$,
d'où l'assertion d'après le lemme~\ref{lemm.div(f)-local}.
\end{proof} 

\begin{prop}
\begin{enumerate}
\item
La formule $f\mapsto \div(f)$ définit 
un morphisme de G-faisceaux en monoïdes abéliens
du G-faisceau des fonctions analytiques régulières sur~$X$
dans le G-faisceau des cycles effectifs sur~$X$.
\item
Ce morphisme de faisceaux se prolonge de manière
unique en un morphisme de G-faisceaux en groupes abéliens
du G-faisceau des fonctions méromorphes régulières sur~$X$
dans le G-faisceau des cycles sur~$X$.
\end{enumerate}
\end{prop}
\begin{proof}
(1) Soit $V$ et $W$ des domaines analytiques de~$X$ tels que $W\subset V$.
Pour toute fonction analytique régulière~$f$ sur~$V$,
$f|_W$ est une fonction analytique régulière sur~$W$ 
et $\div(f|_W)=[\{f|_W=0\}]=[\{f=0\}]|_W=\div(f)|_W$.
Cela prouve que cette formule définit un morphisme de G-faisceaux;
d'après le lemme~\ref{divadd}, c'est un morphisme 
de G-faisceaux en monoïdes.

L'assertion~(2) en découle par définition du faisceau
des fonctions méromorphes sur~$X$.
\end{proof}

\section{Extensions d'Abhyankar et leur défaut}

On fixe un corps ultramétrique complet $k$. 

\subsection{}\label{ss-rappels-abhyankar}
Soit~$F$ une extension de type fini de $k$, de degré de transcendance $d$. 
Soit~$\abs \cdot$ une valeur absolue ultramétrique sur~$F$, 
prolongeant celle de $k$. 
Soit~$X$ un $k$-schéma intègre de corps des fonctions~$F$ 
et soit~$x$ le point de
$X^\an$ situé au-dessus du point générique de~$X$ et 
défini par la valeur absolue~$\abs \cdot$ de $F$. 
Rappelons (\ref{ss-invariants-valuatifs}) que le degré 
de transcendance gradué de~$\widetilde F$ sur~$\widetilde k$ 
est majoré par~$d$ et que l'extension valuée~$F$ de~$k$
est dite d'Abhyankar si on a égalité, 
auquel cas $\widetilde F$ est de type fini sur~$\widetilde k$. 
Comme $\hr x$ est  le complété de~$(F,\abs \cdot)$, 
l'extension valuée~$F$ de~$k$ est d'Abhyankar si et seulement si
$d_k(x)=d$. 

\begin{defi}
Soit $k$ un corps ultramétrique complet 
et~$L$ une extension complète de~$k$.
Nous dirons que $L$ est d'Abhyankar sur~$k$ s'il existe
un espace $k$-analytique~$X$,
un point~$x$ de~$X$ tel que $d_k(x)=\dim_x(X)$
et un $k$-isomorphisme  $\hr x\simeq_k L$ d'extensions complètes de $k$. 
\end{defi}

Si $L$ est une extension complète d'Abhyankar de $k$, 
un triplet $(X,x,\hr x\simeq_k L)$ comme ci-dessus
sera appelé une \textit{présentation analytique} de~$L$ sur~$k$. 
Par abus, nous omettrons le plus souvent
de mentionner l'isomorphisme $\hr x\simeq_k L$ et dirons simplement que $(X,x)$ est une présentation analytique de~$L$ sur~$k$. 

\begin{exem}
Soit $(F,\abs \cdot)$, $X$ et $x$ comme dans \ref{ss-rappels-abhyankar}. 
Si l'extension valuée $F$ de $k$ est d'Abhyankar alors l'extension complète $\widehat F$ de
$k$ est d'Abhyankar, et $(X^\an,x)$ en est une présentation analytique sur $k$. 
\end{exem}

\subsection{}
Soit $L$ une extension complète d'Abhyankar de $k$
et soit $(X,x)$ une présentation 
analytique de~$L$ sur~$k$. 
On a 
\[ \degtr_{\tilde k}(\tilde L) = \degtr_{\tilde k}(\hrt x)=d_k(x)=\dim_x(X).
\]
Nous noterons aussi $d_k(L)$ cet entier.

Comme $d_k(x)=\dim_x X$ on a $\dim_{\mathrm c}(X_\groth,x)=\dim_x X$. 
Par conséquent, si $Y$ est un sous-espace analytique fermé d'un domaine analytique de $X$ tel que
$x\in Y$ alors $\dim_x Y=\dim_x X$, si bien que 
$(Y,x)$ est encore une présentation analytique de~$L$ sur~$k$. En particulier, 
$(X_{\mathrm{red}},x)$ est une présentation analytique de $L$ sur $k$. Si $U$ désigne le lieu normal de 
$X_{\mathrm{red}}$ alors $U$ contient $x$ car l'anneau local artinien réduit $\mathscr O_{X_{\mathrm{red},\groth},x}$
est un corps, et $(U,x)$ est encore une présentation analytique de~$L$ sur~$k$. Choisissons enfin un domaine
affinoïde~$V$ connexe de~$U$ contenant~$x$. 
Le domaine~$V$ est normal, connexe et non vide, et donc intègre ; et le couple
$(V,x)$ est encore une présentation analytique de~$L$ sur~$k$. 
Notons que la dimension de~$V$ est alors égale à~$\dim_x V$, 
donc à~$d_k(L)$.

\subsection{}
Soit $L$ une extension complète d'Abhyankar de $k$.
Une famille $(t_i)$ d'éléments de $L$
 est une \textit{base d'Abhyankar} de $L$ sur $k$ si les $\widetilde {t_i}$ forment une base de transcendance de $\widetilde L$
sur~$\widetilde k$.  
Cela revient à demander que le cardinal de la famille soit égal à $d_k(L)$ et que 
\[\Abs{\sum a_It^I}=\sup_I \abs{a_I}\cdot \abs t^I\] pour toute famille~$(a_I)$ d'éléments presque tous nuls de~$k$. 
Supposons que ce soit le cas et posons $r=(\abs{t_1},\ldots, \abs{t_r})$. 
Il existe alors un unique isomorphisme d'extensions complète de~$k$ de~$\widehat {k(t)}$ sur
$\hr{\eta_r}$ qui envoie~$t_i$ sur~$T_i(\eta_r)$ pour tout~$i$. 

\begin{lemm}\label{lemm.abh-presentation}
Soit $L$ une extension d'Abhyankar de~$k$
et soit $(X,x)$ une présentation analytique de~$L$.
Soit $t=(t_1,\dots,t_n)$ une base d'Abhyankar de~$L$ sur~$k$.
Pour tout~$i$, posons $r_i=\abs{t_i}$.
Soit $\eps>0$.

\begin{enumerate}
\item
Il existe un domaine affinoïde~$V$ de~$X$ contenant~$x$,
des fonctions analytiques~$f_1,\dots,f_n$
telles que $\abs{f_i(x)-t_i}<\eps$ pour tout~$i$.

Supposons désormais que $\eps\leq \abs{t_i}$ pour tout~$i$.

\item
Le point~$f(x)$ de~$\A^n_k$ est égal à~$\eta_r$,
autrement dit, $(f_1(x),\dots,f_n(x))$ est une base d'Abhyankar de~$\hr x$
sur~$k$.

\item
Les sous-corps~$\widehat{k(t)}$ et~$\hr{\eta_r}$ de~$\hr x$ sont 
des extension complètes de~$k$ isomorphes,
les degrés $[\hr x : \widehat{k(t)}]$ et $[\hr x:\hr{\eta_r}]$
sont égaux, et finis.

\item
Les sous-corps résiduels $\hrt{\eta_r}$ et $\widetilde{k(t)}$
de~$\hrt x$ sont égaux.
\end{enumerate}
\end{lemm}
\begin{proof}
Pour démontrer le lemme, on peut remplacer~$X$ par n'importe lequel
de ses domaines analytiques contenant $x$. Cela permet de supposer que
$\dim X=\dim_x X=n$ et que $X$ est affinoïde.
Alors, $X$ est presque bon en~$x$.
Quitte à restreindre encore~$X$, on peut donc 
supposer qu'il existe $n$ fonctions analytiques
$f_1,\ldots, f_n$ sur $X$
telles que $\abs{f_i(x)-t_i}<\eps$ pour tout~$i$.
Cela démontre~\emph a).

Supposons maintenant que $\eps\leq \abs{t_i}$ pour tout~$i$.
Alors $\widetilde{f_i(x)}=\widetilde{t_i}$ pour tout~$i$,
si bien que les~$f_i(x)$ forment une base d'Abhyankar de~$L$ sur~$k$.
Le morphisme $f\colon X\to\A^n_k$ identifie~$\hr{\eta_r}$ au sous-corps
complet de~$\hr x$ engendré par les~$f_i(x)$ sur~$k$.
Il existe de plus un unique $k$-isomorphisme de corps valués
de~$\widehat{k(t)}$ sur~$\hr{\eta_r}$
qui applique $t_i$ sur~$f_i(x)$ pour tout~$i$.
Comme l'extension~$L$ de~$k$ est d'Abhyankar,
le corps gradué~$\hrt L$ est engendré sur~$\tilde k$
par les~$\widetilde{t_i}$;
et $\hrt x$ est engendré sur sur~$\tilde k$ par les~$\widetilde{f_i(x)}$.
Cela prouve~\emph d).

Comme $d_k(\eta_r)=n$ et $\dim(X)=n$, la fibre~$f^{-1}(\eta_r)$
est de dimension nulle.
en particulier, $\hr x$ est finie sur $\hr{\eta_r}$.
L'égalité des degrés dans~\emph c) se 
déduit de
\cite[Theorem 3.1.9]{temkin2021} et \cite[Lemma 6.3.3]{temkin2010}.
\end{proof}

\begin{defi}
Soit $F$ le complété d'une clôture algébrique de $k$. 
Nous dirons qu'une extension complète~$K$ de~$k$ est
\emph{presque algébrique} si elle admet
un $k$-plongement isométrique dans~$F$.
\end{defi}
Cela ne dépend bien sûr pas du choix de~$F$.

\begin{prop}\label{prop.calcul-defaut}
Soit $L$ une extension d'Abhyankar de $k$
et soit $K$ une extension complète presque algébrique de $k$. 
Soit $(X,x)$ une présentation analytique 
de $L$ et soit $\pi \colon X_K\to X$ le morphisme canonique. 
\begin{enumerate}
\item L'anneau $L\widehat{\otimes}_k K$ est artinien. 

\item Pour tout $y\in \pi^{-1}(x)$ l'application composée 
\[L\widehat\otimes_k K\simeq \hr x \widehat\otimes_k K\to \hr y\] 
est surjective.

\item Pour tout $y\in \pi^{-1}(x)$, notons
$\mathfrak m_y$ le noyau de la surjection
de $L\widehat\otimes_k K\to \hr y$ définie ci-dessus. 
L'application $y\mapsto \mathfrak m_y$ établit une bijection entre $\pi^{-1}(y)$ et
$\Spec (L\widehat\otimes_k K)$. 

\item Pour tout $y\in \pi^{-1}(y)$ le corps gradué $\hrt y$ est fini sur 
son sous-corps $\widetilde L\cdot \widetilde K$. 

\item Pour tout $y\in \pi^{-1}(y)$, la longueur de l'anneau 
local artinien $(L\widehat\otimes_k K)_{\mathfrak m_y}$ est égale à la multiplicité
générique de $X_{\mathrm{red}, K}$
le long de l'unique composante irréductible de $X_K$ qui contient $y$. 

\item Supposons que $K$ est le complété d'une extension quasi-galoisienne de $k$. 
Le groupe $\Aut(K/k)$ agit transitivement sur $\Spec (L\widehat\otimes_k K)$. 
\end{enumerate}
\end{prop}
\begin{proof}
Pour démontrer la proposition, on peut supposer que~$X$ est réduit.
On peut également le remplacer par n'importe lequel
de ses domaines analytiques contenant~$x$. 
Cela permet de supposer de plus que
$\dim X=\dim_x X=n$ et que $X$ est affinoïde
et qu'il existe $n$ fonctions analytiques $f_1,\ldots, f_n$ sur~$X$
telles que les $f_i(x)$ forment une base d'Abhyankar de~$\hr x$ sur~$k$.
Dans la suite de la preuve, nous identifions~$L$ à~$\hr x$. 
Pour tout~$i$, $r_i=f_i(x)$; le point~$f(x)$ de~$\A_k^n$ est donc égal à~$\eta_r$.

Comme $d_k(\eta_r)=n$ et que $X$ est de dimension $n$, la fibre $f^{-1}(\eta_r)$ est de dimension nulle ; 
en particulier, $\hr x$ est finie sur $\hr{\eta_r}$.
Comme $f^{-1}(\eta_r)$
est ensemblistement fini, on peut (en remplaçant $X$ par un voisinage
affinoïde convenable de $x$) supposer que $x$ est l'unique antécédent
de $\eta_r$ sur $X$. 
L'espace $X$ étant affinoïde et réduit, il résulte du théorème~6.3.7 
de~\cite{ducros2018} que la fibre $f^{-1}(\eta_r)$ est réduite, 
ce qui veut dire qu'elle s'identifie à $\mathscr M(\hr x)$. 

Notons $\eta_{K,r}$ le point squelettique de $\A^n_K$ de paramètre $r$ ; comme $K$ est presque algébrique
sur $k$, ce point~$\eta_{K,r}$ est l'unique antécédent de~$\eta_r$ 
sur~$\A^n_K$, si bien que $\pi^{-1}(x)=f^{-1}(\eta_{K,r})$
et $\hr{\eta_{K,r}}=\hr{\eta_r}\widehat\otimes_k K$
(corollaire~\ref{coro-skel-ginvariant}, \emph b)).

On a par ailleurs
\begin{align*}
\hr x\widehat{\otimes_k}K&=\hr x \widehat{\otimes_{\hr {\eta_r}}}(\hr{\eta_r}\widehat{\otimes_k}K)\\
&=\hr x \otimes_{\hr {\eta_r}}\hr{\eta_{K,r}},
\end{align*}
d'où il résulte que $\hr x\widehat{\otimes_k}K$ s'identifie à l'anneau des fonctions de la fibre
$f^{-1}(\eta_{K,r})$. C'est par conséquent une $K$-algèbre finie, et en particulier un anneau artinien,
dont
l'ensemble des idéaux maximaux s'identifie \textit{via} les applications d'évaluation 
à la fibre $f^{-1}(\eta_{K,r})=\pi^{-1}(x)$, ce qui montre \emph a),\emph b) et \emph c). 
Soit $y$ un point de $\pi^{-1}(x)$. Le corps gradué $\hrt y$ est fini sur $\hrt{\eta_{K,r}}=\widetilde K\cdot \hrt{\eta_r}$ ; il l'est
a fortiori sur $\widetilde K\cdot \hrt x$, d'où \emph d). 
Le facteur local de $\hr x\widehat{\otimes_k}K$ correspondant à $y$ 
est alors $\mathscr O_{f^{-1}(\eta_{K,r}),y}$, et le morphisme canonique
de $\mathscr O_{X_K,y}$ vers $\mathscr O_{f^{-1}(\eta_{K,r}),y}$
est régulier d'après le théorème~6.3.7 de~\cite{ducros2018}
(parce que $X$ est réduit). 
Le corollaire \ref{coro-longueur-artinien}
assure alors que la longueur de  $\mathscr O_{f^{-1}(\eta_{K,r}),y}$
est égale 
à celle de $\mathscr O_{X_K,y}$, d'où \emph e). 

Supposons maintenant que $K$ est le complété d'une extension quasi-galoisienne de $k$. 
Le groupe $\Aut(K/k)$ agit alors transitivement sur les fibres de $\pi$,  si bien que \emph f)
se déduit de~\emph c). 
\end{proof}

%
%
%
%

\begin{prop}\label{prop.abh-default-gal}
Soit $L$ une extension complète de~$k$ qui est d'Abhyankar.
Soit $K$ une extension complète presque algébrique de~$k$
qui est stable et qui est le complété 
d'une extension algébrique quasi-galoisienne de~$k$. 
Désignons par $\ell$ la longueur de l'anneau artinien $L\widehat\otimes_k K$
et par $\widetilde \ell$ celle de l'anneau gradué
artinien $\widetilde L\otimes_{\widetilde k}\widetilde K$.

Soit $t=(t_1,\dots,t_n)$ une base d'Abhyankar de~$L$ sur~$k$.

L'extension~$L$ de~$\widehat{k(t)}$ est finie et
son défaut est égal à
\[  \ell [\widetilde F:(\widetilde L\cdot \widetilde K)] / \widetilde \ell, \]
pour tout corps résiduel~$F$ de~$L\widehat\otimes_kK$.
\end{prop}
\begin{proof}
Choisissons une présentation analytique~$(X,x)$ de~$L$.
D'après le lemme~\ref{lemm.abh-presentation},
on peut supposer qu'il existe un morphisme $f\colon X\to\A^n_k$
tel que $\abs{f_i(x)-t_i}<\abs{t_i}$ pour tout~$i$.
Alors, les extensions $K/\widehat{k(t)}$ et $K/\widehat{k(f(x))}$
sont finies, ont même degré et même degré résiduel;
elles ont donc même défaut.

On peut ainsi supposer que $t_i=f_i(x)$
pour tout~$i$.
Le point~$f(x)$ de~$\A^n_k$ est alors égal à~$\eta_r$, 
où $r=(\abs{t_1},\dots,\abs{t_n})$.
Le morphisme~$f$ est dimension nulle en~$x$;
quitte à restreindre~$X$, on peut supposer que $f^{-1}(\eta_r)=\{x\}$.

Notons $\pi\colon X_K\to X$ le morphisme de changement de base.

D'après la proposition~\ref{prop.calcul-defaut}, \emph f),
les différents corps 
résiduels~$F$ de l'anneau artinien~$L\widehat\otimes_k K$ 
sont permutés par $\Aut(K/k)$,
de sorte que les différents degrés résiduels
$[\widetilde F:(\widetilde L\cdot\widetilde K)]$ sont égaux
à un même entier~$\widetilde d$.
D'après la proposition~\ref{prop.calcul-defaut}, \emph c), 
ces corps~$F$  sont de la forme~$\hr y$, pour un unique
point $y\in\pi^{-1}(x)$, 
et l'assertion~\emph d) 
affirme que les degrés résiduels considérés sont finis.

Notons $\delta$ le défaut de $\hr x$ sur $\hr{\eta_r}$. 

Il reste à démontrer que l'on a $\delta=\ell \widetilde d/\widetilde \ell$.
On a : 
\[  
\delta=\frac{[\hr x:\hr{\eta_r}]}{[\hrt x :\hrt{\eta_r}]}\\
=\frac{[(\hr x\otimes_{\hr{\eta_r}}\hr{\eta_{K,r}}):\hr{\eta_{K,r}}]}
{[(\hrt x \otimes_{\hrt {\eta_r}}\hrt{\eta_{K,r}}):\hrt{\eta_{K,r}}]}.
\]
Fixons un point $y\in f^{-1}(\eta_{K,r})$ et  posons
$d=[ \hr y : \hr{\eta_{K,r}}]$.

Puisque les points de~$f^{-1}(\eta_{K,r})$ sont permutés
par~$\Aut(K/k)$,
on a
\begin{align*}
[(\hr x\otimes_{\hr{\eta_r}}\hr{\eta_{K,r}}):\hr{\eta_{K,r}}]
& = \sum_{z\in f^{-1}(\eta_{K,r})} [\hr z:\hr{\eta_{K,r}}] \\
& = \ell [ \hr y : \hr{\eta_{K,r}}]  = \ell d. \end{align*}
Comme $K$ est stable,
$\hr{\eta_{K,r}}$ l'est aussi, et
l'entier~$d$ est donc égal à $[\hrt y  :\hrt{\eta_{K,r}}]$.

Par un calcul explicite (il s'agit de corps gradués
de fractions rationnelles, et $\widetilde K$ est algébrique sur 
$\widetilde k$), on a $\hrt {\eta_{K,r}} =\hrt{\eta_r}\otimes_{\widetilde k}
\widetilde K$.
Alors, 
\[
 [(\hrt x {}\otimes_{\hrt {\eta_r}}\hrt{\eta_{K,r}} ):\hrt{\eta_{K,r}}]
 = \widetilde\ell \cdot [ \hrt x\cdot \hrt{\eta_{K,r}}: \hrt{\eta_r}], 
\]
où $\hrt x\cdot \hrt{\eta_{K,r}}$
est le sous-corps de~$\hrt y$ composé de $\hrt x$ et $\hrt{\eta_{K,r}}$,
car $\Aut(\widetilde K/\widetilde k)$ agit transitivement sur
les corps résiduels gradués de l'algèbre artinienne graduée
$\hrt x {}\otimes_{\hrt {\eta_r}}\hrt{\eta_{K,r}}$.

Ainsi,
\begin{align*}
\delta & =\frac{[(\hr x\otimes_{\hr{\eta_r}}\hr{\eta_{K,r}}):\hr{\eta_{K,r}}]}
{[(\hrt x \otimes_{\hrt {\eta_r}}\hrt{\eta_{K,r}}):\hrt{\eta_{K,r}}]} \\
& = \frac{\ell d}{\widetilde \ell \cdot [(\hrt x\cdot \hrt{\eta_{K,r}}): \hrt{\eta_r}]} \\
& = \frac{\ell }{\widetilde \ell} \frac{[\hrt y:\hrt{\eta_{K,r}}]}{[(\hrt x\cdot \hrt{\eta_{K,r}}): \hrt{\eta_r}]}.
\end{align*}
Or, dans $\hrt y$, on a 
\[ \hrt x \cdot \hrt{\eta_{K,r}} = \hrt x \cdot (\hrt{\eta_r}\otimes_{\tilde k} \widetilde K) 
= \hrt x \cdot \widetilde K, \]
de sorte que
\[ 
\frac{[\hrt y:\hrt{\eta_{K,r}}]}{[(\hrt x\cdot \hrt{\eta_{K,r}}): \hrt{\eta_r}]}
 = \frac{[\hrt y:\hrt{\eta_{K,r}}]}{[(\hrt x\cdot \widetilde K): \hrt{\eta_r}]}
 = [\hrt y:(\hrt x\cdot \widetilde K)]. \]
Ainsi,
\[ \delta = \ell 
  [\hrt y:(\hrt x\cdot \widetilde K)] / \widetilde\ell, \]
comme il fallait démontrer.
\end{proof}

\begin{theo}\label{theo.calcul-defaut}
Soit $L$ une extension d'Abhyankar de $k$ et soit $(X,x)$ une présentation
analytique de~$L$. 
Le corps $L$ est une extension finie de $\widehat{k(t)}$
et son défaut ne dépend pas de $t$. 
\end{theo}
\begin{proof}
Le fait que $L$ soit une extension finie de~$\widehat{k(t)}$
a été démontré dans le lemme~\ref{lemm.abh-presentation}.
Soit $K$ le complété d'une clôture algébrique de~$k$;
c'est une extension quasi-galoisienne de~$k$ qui est stable.
Le défaut de~$L$ sur~$\widehat{k(t)}$
a été calculé dans la proposition précédente et
la formule obtenue ne dépend pas de~$t$.
\end{proof}


\begin{defi}
Soit~$L$ une extension d'Abhyankar de~$k$.
On appelle \emph{défaut}  de l'extension d'Abhyankar~$L$
le défaut commun des extensions finies $L/\widehat{k(t)}$,
lorsque $t$ parcourt l'ensemble des bases d'Abhyankar de~$L$ sur~$k$.
On le note $\df_k(L)$.
\end{defi}

Lorsque $L$ est une extension finie de~$k$, c'est une extension d'Abhyankar
avec $d_k(L)=0$, la famille vide étant une base d'Abhyankar de~$L$
sur~$k$ (et même la seule !). 
Le défaut $\df_k(L)$ que nous avons introduit coïncide donc avec
la notion classique.

Soit $X$ un espace $k$-analytique et soit $x$ un point de~$X$
tel que $d_k(x)=\dim_x(X)$. Par définition,
l'extension d'Abhyankar~$\hr x$ de~$k$ est d'Abhyankar;
son défaut est noté~$\df_k(x)$.

\subsection{}\label{ss-blabla-defaut}
%
Soit $L$ une extension complète de~$k$ qui est d'Abhyankar
et soit~$t$ une base d'Abhyankar de~$L$ sur~$k$.

Soit $r=(r_1,\ldots, r_m)$ une famille 
de réels strictement positifs. 
Pour tout corps ultramétrique complet
$F$, notons $F_r$ le complété de 
$F(s_1,\ldots, s_m)$ pour la norme
de Gauss de paramètre $r$ (c'est aussi le corps
$\hr{\eta_{F,r}}$ où $\eta_{F,r}$ est le point
squelettique de $\gm [F]^m$
de paramètre $r$).
Soit~$s$ la base d'Abhyankar standard de $k_r$ sur $k$.

Alors la réunion de $s$ et $t$ est une base d'Abhyankar 
de~$L_r$ sur~$k$, si bien que $t$ est une base d'Abhyankar
de $L_r$ sur $k_r$, et l'on a
$\widehat{k_r(t)}\simeq \widehat{k(t)}_r$. 
Le degré de $L_r$ sur $\widehat{k_r(t)}$ est donc égal
à celui de $L$ sur $\widehat{k(t)}$, et l'on a 
\[\widetilde{L_r}=\widetilde L(\widetilde s)
=\widetilde L\otimes_{\widetilde k(\widetilde t)}
\widetilde k(\widetilde t)(\widetilde s)=\widetilde
L\otimes_{\widetilde k(t)}\widetilde {k_r(t)},\]
si bien que le degré résiduel (gradué) de~$L_r$ 
sur $\widehat{k_r(t)}$ est égal
est égal
à celui de~$L$ sur~$\widehat{k(t)}$. 

Cela démontre l'égalité 
$\df_{k_r}(L_r)=\df_k(L)$; autrement dit,
le défaut de $L_r$ sur $\widehat{k_r(t)}$
est égal à celui de $L$ sur $\widehat{k(t)}$.

\chapter{Calibrage canonique d'un espace analytique}

\section{Cellules commodes et moments agréables}

\begin{defi}\label{defi.f-cellule}
Soit $X$ un espace analytique de niveau~$n$, 
soit $f\colon X\to T$ un moment et soit $C$ une partie de~$X$.
On dit $C$ est une \emph{$f$-cellule} si $C\subset\Sigma_f$
et si c'est une $f$-cellule de l'espace paralinéaire~$\Sigma_f$.
\end{defi}

Toute $f_\trop$-cellule de l'espace vertébré~$X$
au sens de la définition~\ref{defi.f-cellule-trop}
est une $f$-cellule, car $\Sigma_f^{(n)}\subset\Sigma_f$.
Inversement, toute $f$-cellule~$C$ qui est contenue
dans~$\Sigma_f^{(n)}$ est une $f_\trop$-cellule; c'est en particulier
le cas dès que $\dim(C)=n$ ou que $C\subset\Int(X)$.

Une $f$-cellule de~$X$ est une partie squelettique de~$X$;
elle est connexe et non vide, donc est contenue dans
une unique composante connexe de~$X$, et même dans une unique
composante irréductible de dimension~$n$ de~$X$,
car une partie squelettique connexe non vide est
contenue dans une unique composante irréductible de dimension~$n$ de~$X$.

\begin{rema}
Soit $C$ une partie de~$X$ et soit $f\colon X\to\gm^m$ un moment.
Si $f_\trop$ induit un homéomorphisme de~$C$ sur une cellule 
\emph{de dimension~$n$} de~$\R^m$, alors c'est une $f_\trop$-cellule
au sens de la définition~\ref{defi.f-cellule-trop},
donc une $f$-cellule.
\end{rema}

\begin{defi}
Soit $X$ un espace $k$-analytique de niveau~$n$,
soit $f\colon X\to \gm^m$ un moment.
Soit $C$ une $f$-cellule.
On dit que $C$ est \emph{commode}\index{commode} 
si l'une des deux conditions suivantes est vérifiée:
\begin{enumerate}
\item On a $\dim(C)<n$;
\item La cellule~$\mathring C$ est ouverte dans~$\Sigma_f\cap\Int(X)$.
\end{enumerate}
\end{defi}

\begin{remas}\label{remas.cellules-commodes}
\begin{enumerate}
\item
Soit $C$ une $f$-cellule commode.
Supposons que la seconde condition d'une $f$-cellule commode soit satisfaite.
D'après le théorème~\ref{theo.proprietes-Sigmaf},
$\Sigma_f\cap\Int(X)=\Sigma_f^{(n)}$ est purement de dimension~$n$,
si bien que $\mathring C$ est de dimension~$n$.
et $C\subset\Sigma_f^{(n)}$.
En particulier, les deux conditions de la définition sont exclusives
l'une de l'autre.

Notons que l'on a $\mathring C\subset\Sigma_f^{(n)}\cap\Int(X)$,
qui est un ouvert de~$\Sigma_f\cap\Int(X)$. Cependant,
la condition imposée dans la définition est plus forte.

\item
Soit $\Sigma$ une partie squelettique de~$X$
qui est une partie ouverte de~$\Sigma_f$.
Soit~$\mathscr C$ une décomposition cellulaire de~$\Sigma$ en $f$-cellules.
Supposons que l'intérieur~$\mathring C$ de toute cellule~$C\in\mathscr C$ 
de dimension~$n$ soit disjointe du bord~$\partial(X)$ de~$X$.
Alors, toute cellule de~$\mathscr C$ est commode.

En effet, les $n$-cellules ouvertes de~$\mathscr C$
sont ouvertes dans~$\Sigma$, donc dans~$\Sigma_f$.

En particulier, si $X$ est sans bord, toute cellule
d'une décomposition cellulaire en $f$-cellules est commode.
\end{enumerate}
\end{remas}

\begin{lemm}\label{lemm.exists-dec-cell-commode}
Soit $X$ un espace analytique de niveau~$n$,
soit $f\colon X\to T$ un moment,
soit $\Sigma$ un sous-espace paralinéaire 
paracompact de~$\Sigma_f$.
Il existe une décomposition cellulaire $\mathscr C$ de~$\Sigma$
en $f$-cellules commodes.
\end{lemm}
Rappelons que l'hypothèse de paracompacité est nécessaire,
car tout espace paralinéaire qui possède une décomposition
cellulaire est paracompact.


\begin{proof}
Posons $P=\partial(X)\cap\Sigma$ et démontrons d'abord
que $P$ est contenu dans un sous-espace paralinéaire 
de dimension~$<n$.
Comme $\Sigma$  est paracompact, il existe un ensemble~$\mathscr V$ 
de domaines analytiques compacts de~$X$ tel que la famille $(V\cap\Sigma)$ soit un G-recouvrement localement fini de~$\Sigma$.
Pour tout $V\in\mathscr V$, on a 
comme $P\cap V \subset\partial(V)\cap\Sigma$,
il existe un sous-espace paralinéaire~$Q_V$ de dimension~$<n$,
fermé et contenu dans~$V$,
tel que $P\cap V\subset Q_V$ (lemme~\ref{lemm.bord-squelette-negligeable}).
La réunion~$Q$ de ces sous-espaces~$Q_V$, lorsque $V$ parcourt~$\mathscr V$,
est un sous-espace paralinéaire de dimension~$\leq n-1$ de~$\Sigma$.

Comme $\Sigma$ est connexe et paracompact,
il est dénombrable à l'infini.
D'après le corollaire~\ref{coro.voisinage},
il possède un voisinage ouvert~$\Sigma'$ dans~$\Sigma_f$ 
qui est encore dénombrable à l'infini, et en particulier paracompact.

L'application $f_\trop\colon \Sigma_f\to T_\trop$ est une immersion
par morceaux.
Considérons alors une décomposition cellulaire~$\mathscr C'$ de~$\Sigma'$,
adaptée à~$Q$ et à~$\Sigma$,
telle que $(\mathscr C,(f|_C))$ soit un atlas cellulaire
(proposition~\ref{prop.decomp-cell} et remarque~\ref{rema.decomp-cell-f}).
En particulier, les cellules de~$\mathscr C'$ sont des $f$-cellules.
D'après la remarque~\ref{remas.cellules-commodes}, \emph b), 
toutes les cellules de~$\mathscr C'$ sont commodes.

Les cellules de~$\mathscr C'$ qui sont contenues
dans~$\Sigma$ constituent la décomposition cellulaire~$\mathscr C$
cherchée.
\end{proof}

\begin{coro}
Soit $f\colon X\to T$ un moment et
soit $\Sigma$ un sous-espace paralinéaire 
de~$\Sigma_f$.
Les $f$-cellules de~$\Sigma$ qui sont commodes
forment un G-recouvrement de~$\Sigma$.
\end{coro}
\begin{proof}
Comme tout point de~$X$ possède une base d'ouverts paracompacts,
on peut supposer que $X$ est paracompact.
Alors, $\Sigma_f$ est paracompact, donc $\Sigma$ est paracompact.
Il possède donc une décomposition
cellulaire en $f$-cellules commodes,
et elles constituent un G-recouvrement de~$X$.
\end{proof}

\begin{defi}\label{defi.moment-agreable}
Soit $X$ un espace analytique de niveau~$n$.
On dit qu'un moment $f\colon X\to T$ est \emph{agréable}\index{moment agréable} 
si la carte G-tropicale $(X,f_\trop)$ est agréable.
\end{defi}
 
Rappelons (définition~\ref{defi.carte-agreable}) que cela signifie
que 
$f_\trop(\Sigma_f^{(n)})$ est un sous-espace paralinéaire de~$T_\trop$
et que la restriction de~$f_\trop$ à~$\Sigma_f^{(n)}$ 
est compacte sur son image.

Tout moment propre est agréable.

Pour qu'un moment~$f$ soit agréable,
il suffit que $\Sigma_f^{(n)}$ possède une décomposition
cellulaire finie en $f$-cellules. Les deux propositions
suivantes fournissent deux exemples importants.

\begin{prop}\label{prop.compact-moment-agreable}
Si $X$ est un espace $k$-analytique compact de niveau~$n$,
tout moment sur~$X$ est  agréable.
\end{prop}
\begin{proof}
C'est un cas particulier du~\S\ref{exem.compact-agreable}.
\end{proof}

\begin{prop}\label{prop.algebrique-moment-agreable}
Soit $X$ un $k$-schéma de type fini de dimension~$\leq n$;
on considère l'espace~$X^\an$ comme un espace analytique de niveau~$n$.
Soit $T$ un $k$-tore et soit $f\colon X\to T$ un morphisme de $k$-schémas.
\begin{enumerate}
\item On a $\Sigma_{f}^{(n)}=\Sigma_{f}$ 
et $f_\trop(X^\an)^{(n)}=f_\trop(\Sigma_f)$.
\item Le sous-espace paralinéaire $\Sigma_{f}^{(n)}$  de~$X^\an$
possède une décomposition cellulaire finie en $f^\an$-cellules commodes.
\item En particulier, le moment $f\colon X^\an\to T^\an$ est agréable.
\end{enumerate}
\end{prop}
Notons que, $X^\an$ étant sans bord, 
$\Sigma_{f}^{(n)}$ coïncide avec~$\Sigma_f$
sur toute composante irréductible de dimension~$n$, et est vide ailleurs.
\begin{proof}
On sait que $f_\trop(X^\an)$ est un sous-espace paralinéaire de~$T_\trop$,
équidimensionnel; si cette dimension~$m$ vérifie $m<n$,
alors $\Sigma_f$ est vide et il n'y a rien à démontrer.
Supposons que $f_\trop(X^\an)$ soit de dimension~$n$ et démontrons
que l'on a $f_\trop(X^\an)=f_\trop(\Sigma_f)$.
L'inclusion $f_\trop(\Sigma_f)\subset f_\trop(X^\an)$ est évidente.

Il existe un ouvert dense~$U$ de~$X$, 
pour la topologie de Zariski, et un moment $g \colon U\to \gm^{m}$
tel que $g_\trop$ soit injective en restriction à~$\Sigma_f$
(\cite{ducros2012b}, proposition~2.7).
Il existe d'autre part un ouvert dense~$U'$ de~$U$, 
toujours pour la topologie de Zariski,
et un moment $g'\colon U'\to \gm^m$ qui est une immersion fermée.
(Les ouverts « très affines » forment une base d'ouverts de~$X$.)
Notons $h$ le moment $(f,g,g')$ sur $U'$ et $q$
la projection de $T\times \gm^m\times \gm^{m'}$ sur~$T$;
on a $f=q\circ h$. Par construction, $h$ est une immersion fermée
et $h_\trop$ est injective sur~$\Sigma_f$.

On a $\Sigma_f=\Sigma_{f|_{U'}}$,
et $h_\trop(U^{\prime\an})=h_\trop(X^\an)$ 
et $f_\trop(U^{\prime\an})=f_\trop(X^\an)$
(cf. \cite{ducros2012b}, théorème~1.2, assertions~A et~B).
On peut donc remplacer $X$ par~$U'$; 
posons aussi $S=T\times\gm^{m}\times \gm^{m'}$.

Comme $X^\an$ est sans bord, les espaces paralinéaires $\Sigma_f$ et $\Sigma_h$
sont purement de dimension~$n$. 

Le sous-espace paralinéaire~$\Sigma_f$ de~$\Sigma_h$
est l'ensemble des points en lesquels
$f_\trop|_{\Sigma_h}$ est de dimension tropicale~$n$.

Comme $h$ est une immersion fermée, $h^\an$ est propre;
comme la tropicalisation standard d'un tore est un morphisme
topologiquement propre, on en déduit que $h_\trop\colon X^\an\to  S_\trop$ 
est  topologiquement propre.
Par suite, $h_\trop(X^\an)^{(n)}=h_\trop(\Sigma_h)$ (théorème~\ref{theo.proprietes-Sigmaf}, \emph d)).

On sait aussi (cf.~\cite{ducros2012b}, loc. cit.) 
que $h_\trop(X^\an)$ est un sous-espace
paralinéaire fermé de~$S_\trop$, purement de dimension~$n$,
et qu'il possède une décomposition cellulaire finie~$\mathscr C$.

La famille des cellules~$C\in\mathscr C$
sur lesquelles $q_\trop$ est injective est une décomposition
cellulaire d'un sous-espace paralinéaire $h_\trop(X^\an)'$ de~$h_\trop(X^\an)$.

Comme $h_\trop$ est une immersion par morceaux et $f=q\circ h$,
le morphisme $f_\trop|_{\Sigma_h}$ est de dimension
tropicale~$n$ en un point~$x$ de~$\Sigma_h$
si et seulement si $q_\trop|_{h_\trop(X^\an)}$
est de dimension tropicale~$n$ en~$h_\trop(x)$.
Cela prouve que $\Sigma_f=h_\trop^{-1}(h_\trop(X^\an)')$.
Comme $h_\trop|_{\Sigma_f}$ est injective, 
il en résulte que $h_\trop$ induit un isomorphisme
paralinéaire de~$\Sigma_f$ sur~$h_\trop(X^\an)'$. 

On a donc $f_\trop(X^\an) = f_\trop(X^\an)^{(n)} =  q_\trop \circ h_\trop(X^\an)'
= q_\trop(h_\trop(\Sigma_f))=f_\trop(\Sigma_f)$,
ce qui conclut la démonstration de l'assertion~\emph a).

Démontrons~\emph b).
Dans la décomposition cellulaire finie de $h_\trop(X^\an)$
introduite plus haut, la famille des cellules qui sont contenues
dans~$h_\trop(X^\an)'$ constitue une décomposition cellulaire de $h_\trop(X^\an)'$.
On en déduit, par $h_\trop$, une décomposition cellulaire finie
de~$\Sigma_f$.
D'après la remarque~\ref{remas.cellules-commodes}, \emph b), 
ses cellules sont toutes commodes, ce qui établit l'assertion~\emph b),
et l'assertion~\emph c) en résulte.
\end{proof}

\begin{prop}\label{prop.moment-loc-agreable}
Soit $X$ un espace $k$-analytique  de niveau~$n$
et soit $f\colon X\to T$ un moment sur~$X$. 
Tout point de~$X$ possède un voisinage \emph{ouvert}~$U$
tel que $f|_U$ soit un moment agréable.
\end{prop}
\begin{proof}
Soit $x\in X$. 
Si $x\not\in\Sigma_f^{(n)}$, le complémentaire $X\setminus \Sigma_f^{(n)}$
convient. On suppose dans la suite que $x\in\Sigma_f^{(n)}$.

Soit $V_1$ un sous-espace paralinéaire compact de~$\Sigma_f^{(n)}$
qui est un voisinage de~$x$ dans~$\Sigma_f^{(n)}$.
La fibre $V_1\cap f_\trop^{-1}(f_\trop(x))$  de~$f_\trop(x)$ sur~$V_1$
est finie;
soit alors $V_2$ un voisinage paralinéaire compact de~$x$
dans~$\mathring V_1$ tel que $x$ soit le seul antécédent de~$f_\trop(x)$
dans~$V_2$. Notons $\partial(V_2)$ le bord de~$V_2$ dans~$V_1$;
c'est un sous-espace paralinéaire fermé de~$V_2$ 
qui ne contient pas~$x$.

Soit $A=f_\trop(V_2)\setminus f_\trop(\partial(V_2))$
et soit $W=f_\trop^{-1}(A)\cap V_2$.
L'ensemble~$A$ est ouvert dans~$f_\trop(V_2)$,
donc $W$ est ouvert dans~$V_2$; en outre, $W$ ne rencontre pas~$\partial(V_2)$,
donc il est ouvert dans~$V_1$. Comme $V_2\subset\mathring V_1$,
$W$ ne rencontre pas non plus le bord de~$V_1$
dans~$\Sigma_f^{(n)}$, donc est ouvert dans~$\Sigma_n^{(f)}$.

L'application~$f_\trop$ de~$W$ dans~$A$  est compacte et son
image est un sous-espace paralinéaire de~$T_\trop$.
Soit $U$ une partie ouverte de~$X$ telle que $U\cap\Sigma_f^{(n)}=W$.
On a  $\Sigma_{f|_U}^{(n)}=\Sigma_f^{(n)}\cap U=W$,
si bien que le moment $f|_U$ est agréable.
\end{proof}


\begin{rema}
Dans les deux cas précédents (compact,
proposition~\ref{prop.compact-moment-agreable}, et algébrique,
proposition~\ref{prop.algebrique-moment-agreable}),
les polyèdres caractéristiques
jouissent d'une propriété de modération particulièrement forte:
on peut remplacer $\Sigma_f^{(n)}$  par $\Sigma_f$ et
la décomposition cellulaire peut être choisie finie;
dans le cas algébrique, le bord est même vide.
\end{rema}

\begin{prop}\label{prop.agreable-decomp}
Soit $X$ un espace analytique de niveau~$n$
et soit $f\colon X\to T$ un moment agréable.
Alors, $\Sigma_f^{(n)}$ est paracompact.

Il existe une décomposition cellulaire~$\mathscr C$
de~$\Sigma_f^{(n)}$ en $f$-cellules commodes
et une décomposition cellulaire forte $\mathscr D$ 
de~$f_\trop(\Sigma_f^{(n)})$ 
tel que le couple 
$(\mathscr C,\mathscr D)$
soit adapté à l'immersion par morceaux compacte $f_\trop|_{\Sigma_f^{(n)}}$.

Enfin, pour tout couple $(\mathscr C,\mathscr D)$
comme ci-dessus, 
$f_\trop(\partial(X)\cap \Sigma_f^{(n)})$ est une partie
fermée de~$f_\trop(\Sigma_f^{(n)})$, contenue dans la réunion des cellules
de dimension~$\leq n-1$ de~$\mathscr D$.
\end{prop}

On dira qu'une décomposition cellulaire~$\mathscr C$
de~$\Sigma_f^{(n)}$ est \emph{agréable} (relativement à~$f$) s'il existe
une décomposition cellulaire~$\mathscr D$ (nécessairement unique)
telle que le couple $(\mathscr C,\mathscr D)$ vérifie les conclusions
de la proposition. 

\begin{proof}
L'espace $f_\trop(\Sigma_f^{(n)})$ est réunion d'une famille
dénombrable de sous-espaces compacts, dont l'image
réciproque par $f_\trop$ est compacte dans~$\Sigma_f^{(n)}$,
par définition d'un moment agréable. Par suite, $\Sigma_f^{(n)}$
est réunion d'une famille dénombrable de sous-espaces compacts.
Cela entraîne qu'il est paracompact.

D'après le lemme~\ref{lemm.exists-dec-cell-commode}, il existe
une décomposition cellulaire~$\mathscr C_1$ de~$\Sigma_f^{(n)}$
en $f$-cellules commodes.
D'après la proposition~\ref{prop.decomp-relative},
il existe une décomposition cellulaire~$\mathscr C$ de~$\Sigma_f^{(n)}$
qui raffine~$\mathscr C_1$ et une décomposition
cellulaire $\mathscr D$  adaptée à l'inclusion de~$f_\trop(\Sigma_f^{(n)})$
dans~$T_\trop$ tel que le couple $(\mathscr C,\mathscr D)$
soit adapté à l'immersion par morceaux compacte $f_\trop$.
Ses cellules sont des $f$-cellules commodes 
car la décomposition cellulaire~$\mathscr C$
raffine~$\mathscr C_1$. 

La dernière assertion découle de la définition
d'une cellule commode.
%
\end{proof}

\section{Moments et degrés}

\begin{lemm}\label{lemm.quasi-fini-plat-sq}
Soit $X$ un espace analytique de niveau~$n$,
soit $f\colon X\to\gm^n$ un moment.
Alors, $f$ est quasi-fini et plat en tout point de~$\Sigma_f$.
\end{lemm}
\begin{proof}
Soit $x\in\Sigma_f$. On a $f(x)\in S(\gm^n)$, donc $d_k(f(x))=n$.
D'après le lemme~\ref{lemm.fibre-dim-zero}, le morphisme~$f$
est quasi-fini en~$x$.
Comme $\gm^n$ est réduit, le théorème~10.3.7 de~\cite{ducros2018}
entraîne que $f$ est plat en~$x$.
\end{proof}

\begin{prop}\label{prop.commode-fini-plat}
Soit $X$ un espace analytique de niveau~$n$,
soit $f\colon X\to T$ un moment sur~$X$
et soit $C$ une $f$-cellule commode de dimension~$n$ de~$X$.
Soit $q\colon T\to\gm^n$ un morphisme de tores
tel que $q_\trop|_{f_\trop(C)}$ est injectif.
Alors le morphisme $q \circ f$ est de dimension relative nulle
en tout point de~$C$
et il existe un unique entier~$d>0$ tel que ce morphisme soit fini et plat 
de degré~$d$ en tout point de~$\mathring C$.
\end{prop}
\begin{proof}
Soit $q\colon T\to \gm^n$ un morphisme de tores tel que $q_\trop$
soit injectif sur $f_\trop(C)$.
La dimension tropicale de~$q\circ f$ est égale à~$n$
en tout point de~$C$, 
donc $C\subset (q\circ f)^{-1} (S(\gm^n))$ 
d'après la prop.~\ref{prop.trop.local}, (iv).
D'après le lemme~\ref{lemm.fibre-dim-zero},
le morphisme~$q\circ f$ est de dimension relative nulle 
en tout point de~$C$.

Comme $\mathring C$ ne rencontre pas~$\partial(X)$, 
il résulte du lemme~\ref{lemme.finiplat-local}
que le morphisme~$q\circ f$ est fini et plat en tout point de~$\mathring C$;
on a noté $\deg^x(q\circ f)$ son degré et 
il s'agit de démontrer que l'entier~$\deg^x({q\circ f})$
est indépendant du point~$x\in \mathring C$.

Soit $D$ la cellule~$\sq(q_\trop(f_\trop(C)))$ de~$\gm^n$.
Si un point $x\in\Int(X)$ vérifie $q\circ f(x)\in \mathring D$,
alors $q\circ f$ est de dimension tropicale~$n$ en~$x$,
donc $f$ aussi, si bien que $x\in\Sigma_f^{(n)}$.
Autrement dit, on a les inclusions
\[  \mathring C\subset(q\circ f)^{-1}(\mathring D)\cap \Int(X)
  \subset \Sigma_f^{(n)}.\]

Soit maintenant $x\in\mathring C$.
Par hypothèse, $\mathring C$ est ouvert dans~$\Sigma_f^{(n)}$;
il est donc ouvert dans $(q\circ  f)^{-1}(\mathring D)\cap\Int(X)$.
Comme $q\circ f$ est fini et plat en~$x$,
il existe un voisinage ouvert~$U$ de~$x$ dans~$\Int(X)$
et un voisinage ouvert~$V$ de~$q\circ f(x)$ dans~$\gm^n$
tels que $q\circ f$ induise un morphisme fini et plat
de degré $\deg^x ({q\circ f})$ de~$U$ sur~$V$,
et tel que $(q\circ f)^{-1}(\mathring D)\cap U\subset \mathring C$.
Un point $y\in\mathring C\cap U$ est alors le seul antécédent dans~$U$
de $q\circ f(y)$ par~$q\circ f$; 
on a donc  $\deg^y({q\circ f})=\deg^x({q\circ f})$
pour tout $y\in \mathring C\cap U$. La fonction~$x\mapsto \deg^x({q\circ f})$
est localement constante sur l'ensemble~$\mathring C$;
comme celui-ci est connexe, elle est constante.
\end{proof}

\begin{prop}\label{prop.dec-cell-deg}
Soit $X$ un espace $k$-analytique paracompact de niveau $n$
et soit $f\colon X\to \gm^n$ un moment. 
Soit $\Sigma$ un sous-espace paralinéaire paracompact de~$\Sigma_f$.
Il existe une décomposition cellulaire
$\mathscr C$ de~$\Sigma$ en $f$-cellules telle que pour toute
cellule $C\in \mathscr C$,
l'application $x\mapsto \deg^x(f)$ soit finie et constante sur~$\mathring C$.
\end{prop}
Rappelons que le degré~$\deg^x(f)$ d'un morphisme~$f$ en un point~$x$
a été défini au \S\ref{ss.degre} sans hypothèse de finitude ou de platitude.
Il peut éventuellement être infini.
\begin{proof}
Pour tout $x\in\Sigma_f$, on a $d_k(x)=n$, donc $f$ est quasi-fini
en~$x$ (lemme~\ref{lemm.fibre-dim-zero}); cela prouve que $\deg^x(f)$ est fini.

Nous procédons alors par récurrence sur le niveau $n$, 
et supposons donc la proposition vraie en niveau $<n$. 
Puisque $\Sigma$ est paracompact, 
il possède une décomposition cellulaire~$\mathscr D$ 
constituée de $f$-cellules commodes
(lemme~\ref{lemm.exists-dec-cell-commode}). 

Soit $D$ une cellule de $\mathscr D$.

Si $D$ est de dimension~$n$, la proposition~\ref{prop.commode-fini-plat}
assure que $f$ est fini et plat de degré constant sur~$\mathring D$;
en particulier, la fonction $x\mapsto \deg^x (f)$ est constante
sur~$\mathring D$.

Supposons maintenant que $\dim(D)<n$.
Il existe un domaine analytique compact~$V$ de $X$ dont l'intersection 
avec~$\Sigma$ est égale à~$V$ 
(lemme \ref{lemme-proprietes-paralineairesX} (e)) ; 
on a donc $D=\Sigma_{f|V}$. 
Comme $D$ est une $f$-cellule, $f(D)$ est une partie du 
squelette canonique~$S(\gm^n)$, 
et $f(D)_\trop$ est une $d$-cellule forte $\Delta$ de $\R^n$. 
Il existe un automorphisme~$\phi$ de~$ \gm^n$  tel que
l'espace affine engendré par $\phi_\trop(\Delta)$ soit défini par les $n-d$
égalités $x_{d+1}=\cdots=x_{n}=0$. Quitte à changer~$f$
en~$\phi\circ f$, on se ramène au cas où $\phi$ est l'identité.
La cellule $D$ est ainsi contenue 
dans $f^{-1}(S(\gm^n)\cap p^{-1}(\eta_{1}))$, 
où $p$ désigne la projection de~$\gm^n$ sur les $n-d$ dernières composantes
et où $\eta_1$ désigne, par abus, le point $\eta_{1,\dots,1}$ de~$\gm^{n-d}$.


Notons~$V'$ la fibre $(p\circ f)^{-1}(\eta_{1})\cap V$. 
Puisque $d_k(\eta_{1})=n-d$, 
la dimension de~$V'$ est au plus~$n-(n-d)=d$; on le considère comme un espace 
$\hr{\eta_{1}}$-analytique de niveau~$d$.
Notons $f'\colon V'\to \gm[\hr{\eta_1}]^d$ le morphisme 
d'espaces $\hr{\eta_1}$-analytiques donné par $(f_1,\dots,f_d)$.
On a $D=\Sigma_{f'}$. En effet, pour $x\in V'$,
on a $x\in\Sigma_{f'}$ si et seulement si $\widetilde{f_1(x)},
\dots,\widetilde{f_d(x)}$ sont algébriquement indépendants sur~$\hrt{\eta_1}$;
comme ce corps est transcendant pur engendré par les $\widetilde{f_i(x)}$
pour $i>d+1$, cela revient à exiger
que $\widetilde{f_1(x)},\dots,\widetilde{f_n(x)}$ soit algébriquement
indépendant sur~$\tilde k$, c'est-à-dire $x\in \Sigma_f$,
soit encore $x\in D$ puisque $D=\Sigma_f\cap V$.

Puisque $d<n$, 
l'hypothèse de récurrence assure l'existence 
d'une décomposition cellulaire~$\mathscr D_D$ de~$D$
telle que $x\mapsto \deg^x(f')$ soit constant sur toute cellule ouverte 
de~$\mathscr D_D$. 
Comme les fibres de~$f'$ sont celles de~$f$, on a $\deg^x(f')=\deg^x(f)$
pour tout $x\in D$.
Par conséquent, $x \mapsto \deg^x(f)$ est constant 
sur toute cellule ouverte de~$\mathscr D_D$. 

Choisissons maintenant une décomposition cellulaire~$\mathscr D'$ 
de la réunion des cellules~$D$ de~$\mathscr D$ de
dimension $<n$ et qui raffine chacune des $\mathscr D_D$. 

La famille~$\mathscr C$ formée des cellules de~$\mathscr D'$
et des cellules de~$\mathscr D$ qui sont de dimension~$n$
est une décomposition cellulaire de~$\Sigma$. Elle répond
aux exigences de l'énoncé.
\end{proof}

\subsection{}
Conservons les notations de la proposition précédente
et notons $d_C$ la valeur constante de $\deg^x(f)$ sur~$\mathring C$.
Comme $x\mapsto \deg^x(f)$ est semi-continue supérieurement,
on a $d_D\leq d_C$ si $C\subset D$.
Lorsque $X$ est sans bord, on a des relations
plus précises données par le théorème suivant qui montre,
en outre, que toute décomposition cellulaire de~$\Sigma_f$
convient.

\begin{theo}\label{theo.df-recurrence}
Soit $X$ un espace $k$-analytique de niveau~$n$, sans bord,
soit $f\colon X\to \gm^n$ un moment sur~$X$ et
soit $\mathscr C$ une décomposition cellulaire de~$\Sigma_f$
en $f$-cellules.
	
\begin{enumerate}
\item
Pour toute cellule~$C$ de~$\mathscr C$, il existe un unique entier~$d_C>0$
tel que $f$ soit fini et plat de degré~$d_C$ en tout point de~$\mathring C$.
\item
Soit $C\in\mathscr C$ et soit $d$ un entier tel que $\dim(C)\leq d\leq n$.
Soit $\Delta$ un ensemble non vide de cellules de~$\mathscr C$
qui contiennent~$C$, tel que $\bigcap_{D\in\Delta} f_\trop(D)$
soit une cellule de dimension~$d$, et maximal pour cette propriété;
notons $\Delta^*$ l'ensemble de ses éléments minimaux.
On a
\[ d_C = \sum_{D\in \Delta^*} d_D . \]
\end{enumerate}
\end{theo}
\begin{proof}
Sous les hypothèses du théorème, toute cellule de~$\mathscr C$
est une $f$-cellule commode (remarque~\ref{remas.cellules-commodes}, \emph b)).
D'après le lemme~\ref{lemme.finiplat-local}, le morphisme~$f$
est fini et plat en tout point de~$\Sigma_f$;
le contenu du théorème est la constance du degré sur chaque cellule ouverte
et une formule de récurrence pour ces degrés.

On démontre ces résultats par récurrence descendante sur $c=\dim(C)$.
Lorsque $c=n$, 
la première assertion découle de la proposition~\ref{prop.commode-fini-plat},
et la seconde revient à l'égalité triviale $d_C=d_C$,
car on a alors nécessairement $d=n$ et $\Delta=\Delta^*=\{C\}$.

Soit $c$ un entier~$<n$;
supposons les assertions~(1) et~(2) vérifiées 
pour toute cellule~$C$ de dimension~$>c$.
Pour une telle cellule $C$, on a donc un degré $d_C$ bien défini.

Soit $C\in\mathscr C$ une cellule de dimension~$c$.
Soit $d$ un entier tel que $c\leq d\leq n$.
Soit $\mathscr E_{C,d}$ l'ensemble des parties~$\Delta$ de~$\mathscr C$
formées de cellules contenant~$C$ et telles que
$\bigcap_{D\in\Delta} f_\trop(D)$ soit de dimension~$d$,
et maximale pour cette propriété.
Pour $\Delta\in\mathscr E_{C,d}$,
notons $\Delta^*$ l'ensemble des éléments minimaux de~$\Delta$,
lorsque ce dernier est ordonné par l'inclusion.

Comme $\Sigma_f$ est purement de dimension~$n$, il existe 
une $n$-cellule de~$\mathscr C$ qui contient~$C$.
Il existe alors une cellule $D\in\mathscr C$ de dimension~$d$ qui
contient~$C$. Comme $f_\trop(D)$ est de dimension~$d$,
il existe  une partie $\Delta$ appartenant à~$\mathscr E_{C,d}$
qui contient~$D$.  En particulier, $\mathscr E_{C,d}$ n'est pas vide.

\begin{lemm}\label{lemm.df-recurrence}
(Sous l'hypothèse de récurrence.)
Supposons $d>c$. Pour tout $\Delta\in\mathscr E_{C,d}$
et tout $x\in\mathring C$, 
on a $\deg^x(f)=  \sum_{D\in\Delta^*}d_D$.
\end{lemm}

Admettons provisoirement ce lemme.

Rappelons que $c<n$; appliqué à $d=n$ et à un élément
de l'ensemble non vide~$\mathscr E_{C,n}$,  le lemme~\ref{lemm.df-recurrence}
entraîne
que la fonction $x\mapsto \deg^x(f)$ est constante sur l'ensemble
non vide~$\mathring C$; notons $d_C$ cette valeur.
Cela établit déjà l'assertion~(1) pour~$C$.

Soit $d$ un entier tel que $d\leq c\leq n$.
Le lemme~\ref{lemm.df-recurrence} établit alors également l'assertion~(2) dès que $d>c$.
Dans le cas restant où $d=c$, on a 
$\Delta=\Delta^*=\{C\}$ et l'assertion~(2) se réduit alors à l'égalité 
tautologique $d_C=d_C$.

Cela termine la démonstration du théorème et
il reste à démontrer le lemme~\ref{lemm.df-recurrence}
que nous avions admis.

Notons $W_C$ la réunion des cellules ouvertes~$\mathring D$,
où $D$ parcourt l'ensemble des cellules de~$\mathscr C$ contenant~$C$;
c'est un voisinage ouvert de~$\mathring C$ dans~$\Sigma_f$.
Soit $V$ un voisinage ouvert de~$x$ dans~$X$ tel que $V\cap\Sigma_f\subset W_C$
et tel que $f$ induise un morphisme fini et plat de degré~$\deg^x(f)$
de~$V$ sur un voisinage ouvert~$U$ de~$f(x)$ dans~$\gm^n$.

Notons $P$  l'intersection des $f_\trop(D)$, pour $D\in\Delta$;
c'est une cellule de dimension~$d$ contenant $f_\trop(C)$.
Soit $E\in\mathscr C$ contenant~$C$ telle que $E\not\in\Delta$;
alors $\Delta\cup\{E\}$ n'appartient pas à~$\mathscr E_{C,d}$, par maximalité
de~$\Delta$,
de sorte que $P\cap f_\trop(E)$ est une cellule de dimension~$<d$.
Notons~$Q_1$ la réunion de ces $P\cap f_\trop(E)$,
lorsque $E$ parcourt l'ensemble des cellules de~$\mathscr C$
contenant~$C$ et n'appartenant pas à~$\Delta$;
c'est un sous-espace paralinéaire fermé de~$P$, de dimension~$<d$.

D'autre part, soit $D\in\Delta^*$ et soit $E$ une cellule contenue
dans~$\partial(D)$ disjointe de~$\mathring C$;
alors $P\cap f_\trop(E)$ est une cellule, disjointe de $f_\trop(\mathring C)$,
car $f_\trop$ est injectif sur~$D$.
La réunion~$Q_2$ de ces cellules $P\cap f_\trop(E)$
est un sous-espace paralinéaire de~$P$, fermé, 
disjoint de~$f_\trop(\mathring C)$, de sorte que $P\setminus Q_2$
est un voisinage de~$f_\trop(\mathring C)$ dans~$P$.

Posons $Q=Q_1\cup Q_2$; c'est un sous-espace paralinéaire de~$P$,
fermé.
L'ensemble $P\setminus Q_2$ est un voisinage de $f_\trop(\mathring C)$ dans~$P$;
choisissons un ouvert~$\Omega$ de~$P$ tel que $f_\trop(\mathring C)\subset \Omega
\subset P\setminus Q_2$.
Comme $Q_1$ est d'intérieur vide dans~$P$, sa trace
sur~$\Omega$ est d'intérieur vide dans~$U$,
de sorte que $U\setminus Q_1$ est dense dans~$\Omega$.
Cela prouve que l'adhérence de $P\setminus Q$ contient~$\Omega$;
elle contient en particulier $f_\trop(\mathring C)$,
puis $f_\trop(C)$ qui est son adhérence.

Soit $P'$ l'ensemble des points~$\xi\in P$
tels que d'une part $\sq(\xi)\in U$, et  d'autre part,
pour toute cellule $D\in\Delta$,
l'unique antécédent de~$\xi$ par~$f_\trop|_D$ appartient à~$V$.
C'est un ouvert de~$P$ qui contient $f_\trop(x)$.
Il rencontre donc $P\setminus Q$.
Soit alors $u\in U\cap (P'\setminus Q)_\sq$.

Soit $v\in V\cap f^{-1}(u)$. Comme $f(v)=u\in S(\gm^n)$, on a $v\in\Sigma_f$;
soit $D$ l'unique cellule de~$\mathscr C$ telle que $v\in \mathring D$.

Démontrons que $D\in\Delta^*$.
Comme $V\cap \Sigma_f\subset W_C$, on a $v\in W_C$, donc $D$ contient~$C$.
Comme $f_\trop(v)\in P\setminus Q_1$, la cellule~$D$ appartient à~$\Delta$.
Soit $E\in\Delta^*$ une cellule minimale contenue dans~$D$;
par hypothèse, il existe $w\in E$ tel que $f_\trop(w)=f_\trop(v)$,
car $f_\trop(v)\in P$; alors $w=v$ car $f_\trop$ est injective sur~$D$;
comme $v\in\mathring D$, cela entraîne $D=E$, et donc $D\in\Delta^*$,
comme annoncé.

Inversement, 
soit $D\in\Delta^*$.
On a $\trop(u)\in f_\trop(D)$, par hypothèse; soit $v\in D$
tel que $f_\trop(v)=\trop(u)$; comme $f(v)$ et $u$ appartiennent à~$S(\gm^n)$,
on en déduit que $f(v)=u$. 
Comme $\trop(u)\in P'$, on a $v\in V$.
Démontrons que $v\in\mathring D$.
Dans le cas contraire, il existerait une cellule~$E$ de~$\mathscr C$,
contenue dans~$\partial(D)$ telle que $v\in E$. 
On a $\trop(u)\in P\cap f_\trop(E)$.
On a $E\not\in\Delta$, par minimalité de~$D$.
Si $E$ contenait~$C$, on aurait $\trop(u)\in  Q_1$, ce qui n'est pas.
Donc $E$ ne contient pas~$C$, donc est disjoint de~$\mathring C$,
et $\trop(u)$ appartient alors à~$Q_2$, ce qui n'est pas.

Cela prouve que la relation « $v\in\mathring D$ » dans~$V\times \mathscr C$
définit une bijection
de l'ensemble $V\cap f^{-1}(u)$ sur l'ensemble $\Delta^*$.
Puisque $f$ induit un morphisme fini et plat de degré~$\deg^x(f)$
de~$V$ sur~$U$,
on a $\deg^x(f) =\deg_{f(x)}(f)
=\deg_u(f)=\sum_{v\in V\cap f^{-1}(u)} \deg^v(f)$.
Par l'hypothèse  de récurrence, 
on a $\deg^v(f)=d_D$ si $v\in\mathring D$ et $D\in\Delta^*$,
car $\dim(D)\geq d>\dim(C)$,
de sorte que $\deg^x(f)=\sum_{D\in\Delta^*} d_D$.

Cela conclut la démonstration du lemme~\ref{lemm.df-recurrence}
et, partant, du théorème~\ref{theo.df-recurrence}.
\end{proof}

\begin{theo}\label{theo.vecteur-volume-transvection-2}
Soit $X$ un espace analytique de niveau~$n$,
soit $f,f'\colon X\to\gm^n$ des moments tels que $f'_\trop-f_\trop$
soit constante sur~$X$.
Les ensembles $\Sigma_f\cap\Int(X)$, $\Sigma_{f'}\cap\Int(X)$,
 $\Sigma_{f_\trop}^{(n)}\cap \Int(X)$ et $\Sigma_{f'_\trop}^{(n)}\cap\Int(X)$
sont égaux.
Pour tout point~$x$ dans cet ensemble, 
les morphismes~$f$ et~$f'$ sont finis et plats en~$x$
et l'on a $\deg^x(f)=\deg^x(f')$.
\end{theo}
\begin{proof}
Les dimensions tropicales de~$f_\trop$ et~$f'_\trop$ sont égales
en tout point; en particulier, 
on a $\Sigma_{f_\trop}^{(n)}=\Sigma_{f'_\trop}^{(n)}$.
La première assertion résulte
donc du corollaire~\ref{coro.Sigmaf=Sigmafn}.
Pour un point~$x\in\Sigma_f\cap\Int(X)$, 
on a $d_k(f(x))=d_k(f'(x))=n$ puisque
ces points appartiennent au squelette de~$\gm^n$.
Comme $\gm^n$ est réduit,
le lemme~\ref{lemme.finiplat-local} entraîne que~$f$ et~$f'$
sont finis et plats en~$x$.
Il reste à démontrer que $\deg^x(f)=\deg^x(f')$.

Soit $K$ le complété d'une clôture algébrique de~$k$; les deux squelettes
canoniques $S(\gm^n)$ et $S(\mathbf G_{\mathrm m,K}^n)$ sont identifiés à~$\R^n$
par la tropicalisation; notons ainsi $y_K$ et~$y'_K$ les points
de $S(\mathbf G_{\mathrm m,K}^n)$ correspondant à~$y$ et~$y'$.
On constate sur les formules définissant les semi-normes correspondantes
que le point $y_K$ est invariant 
sous l'action du groupe $\Aut(K/k)$; 
comme $\gm^n$ est (ensemblistement) le quotient de~$(\gm^n)_K$ par l'action 
de ce groupe, le point~$y_K$
est le seul antécédent de~$y$ dans~$(\gm^n)_K$. De même,
le point~$y'_K$ est le seul antécédent de~$y'$ dans~$(\gm^n)_K$.

Notons $p\colon X_K\to X$ le morphisme de changement de base canonique.
D'après ce qui précède, $p^{-1}(x) \subset f_K^{-1}(y_K)$
et $p^{-1}(x)\subset (f')_K^{-1}(y'_K)$.
Appliquons le lemme~\ref{lemm.degre.K/k}. On obtient
\[ \deg^x(f)  = \sum_{z\in p^{-1}(x)} \deg^z(f_K)
\quad\text{et}\quad
  \deg^x(f')  = \sum_{z\in p^{-1}(x)} \deg^z(f'_K) .\]

Soit $z\in p^{-1}(x)$. Le point~$z$ est intérieur à~$X_K$.
De plus, il appartient à l'image réciproque du squelette canonique
par~$f_K$ et par~$f'_K$.
Quitte à étendre les scalaires à~$K$, on peut ainsi supposer 
que $k$ est algébriquement clos.

Notons $(f_1,\dots,f_n)$ et $(f'_1,\dots,f'_n)$ les composantes de~$f$
et de~$f'$.
Soit $i\in\{1,\dots,n\}$.
La fonction numérique~$\abs{f'_i/f_i}$ est constante  sur~$X$;
comme~$x\in\Int(X)$
et comme le corps~$k$ est maintenant supposé algébriquement clos, 
la proposition~\ref{prop.abs-f-constante}
affirme qu'il existe un élément~$\lambda_i\in k^\times$ 
et une fonction~$g_i\in \mathscr O_{X,x}^\times$ 
tels que 
$\abs {g_i} <\abs{\lambda_i}$ et~$f_i=\lambda_i+g_i$ au voisinage de~$x$. 
Quitte à composer~$f$ avec la translation par $(\lambda_1,\dots,\lambda_n)$
dans~$\gm^n$, on suppose que $\lambda_1=\dots=\lambda_n=1$.
Alors, $\trop(y)=\trop(y')$, donc $y=y'$ puisque ces deux points
appartiennent à~$S(\gm^n)$.

On a~$d_k(x)=n=\dim(X)$;
l'anneau local~$\mathscr O_{X,x}$ est donc artinien (\S\ref{ss.dimcent});
c'est en particulier un anneau de Cohen-Macaulay. 
Choisissons alors un voisinage affinoïde connexe~$V$ de~$y$ 
dans~$\gm^n$ et
et un voisinage affinoïde~$W$ de~$x$ dans~$\Int(X)$ 
tels que les conditions suivantes soient satisfaites : 
\begin{itemize}
\item Les fonctions~$g$ sont définies sur~$W$
et l'on a~$\abs {g_i}<1$ et~$f'_i=(1+g_i)f_i$ sur~$W$ ;
\item L'espace~$W$ est de Cohen-Macaulay ; 
\item Le point~$x$ est le seul antécédent de~$f(x)$ 
par~$f$ dans~$W$;
\item Le morphisme $f$ définit 
par restriction un morphisme fini et plat de~$W$ vers~$V$.
\end{itemize}
Comme~$x$ est le seul antécédent de~$f(x)$ par~$f$ dans~$W$,
on a $\deg_{f(x)}(f)=\deg^{x}(f)$.
Comme $V$ est connexe, il s'ensuit le morphisme
fini et plat $f|_W\colon W\to V$ est de degré constant~$d$.

Soit~$R$ un nombre réel~$>1$ tel que $R \abs g < 1$ sur~$W$
et soit $B=D(0,R)$  le disque ouvert de rayon~$R$ dans la droite
affine ; on  note $\theta$ la fonction coordonnée sur~$B$.
Soit $h$ le morphisme de~$W\times B$ vers~$\gm^n\times B$
défini par la famille
\[((1+\theta g_1)f_1, \dots, (1+\theta g_n)f_n ).\]
Pour tout~$t\in B$, on note~$h_t$ le 
morphisme de~$W_{\hr t}$ dans~$ \mathbf G_{{\mathrm m},\hr t}^n$ 
déduit de~$h$ par restriction à la fibre de la dernière projection en~$t$. 
En particulier, on a $h_0=f|_{W}$ et $h_1=f'|_{W}$. 

Rappelons (\S\ref{section.shilov}) que pour $r\in \R^n$
et $t\in B$, on a noté $\sq_t(r)$ l'unique 
point de~$\gm^n\times_k B$ caractérisé par les propriétés
suivantes:
\begin{itemize}
\item son image par la composition de la première
projection et de l'application de tropicalisation
est~$r$;
\item son image par la seconde projection est~$t$;
\item il appartient au squelette canonique 
du $\mathscr H(t)$-tore~${\mathbf G}^n_{\mathrm m,\mathscr H(t)}$,
immergé dans \mbox{$\gm^n\times B$.}
\end{itemize}

\begin{lemm}\label{lemm.comparaison-squelettes}
Soit $r$ un point de~$\R^n$.
Soit $z$ un point de~$W\times B$, d'image~$w$ dans~$W$
et~$t$ dans~$B$. Si $h_t(z)=\sq_t(r)$, on a $h_0(w)=\sq(r)$.
\end{lemm}
\begin{proof}
Puisque $R\abs g<1$ sur~$W$ et 
puisque~$h_t(z)=\sq_t(r)$, il résulte de
la seconde assertion du \S\ref{imrecisquel-explicite}
que~$h_{0,\hr t (z)}=\sq_t (r)$.
Le point $h_0(w)$ est alors l'image de~$\sq_t(r)$ dans~$\gm^n$,
c'est-à-dire~$\sq(r)$.
\end{proof}

Le point~$x$ appartient à l'intérieur de~$W$ car~$W\subset\Int(X)$,
et  c'est le seul antécédent de~$q\circ f(x)=h_0(x)$
sur~$W$ par~$h_0$. Il existe donc un pavé ouvert~$\Delta$ de~$\R^n$
contenant \mbox{$(q\circ f)_\trop(x)$} 
tel que~$\Delta_\sq$ soit contenu dans~$V$
et ne rencontre pas~$h_0(\partial W)$. 

\begin{lemm}\label{lemm.finiplat-Delta}
Notons $\sq_B(\Delta)$ l'ensemble des points $\sq_t(r)$,
pour $r\in\Delta$ et $t\in B$.
Le morphisme~$h$ est fini et plat au-dessus de $\sq_B(\Delta)$.
\end{lemm}
\begin{proof}
Soit $r\in\Delta$ et $t\in B$; 
soit $z\in W\times B$ un antécédent de~$\sq_t(r)$ par~$h$
et soit $w$ son image dans~$W$. 
D'après le lemme~\ref{lemm.comparaison-squelettes},
on a $h_0(w)=\sq(r)$.
Par le choix de~$\Delta$, cela entraîne que $w\in\Int(W)$;
en conséquence, $z\in(\Int(W))_{\hr t}$.
Nous avons ainsi prouvé que $h^{-1}(\sq_t(r))$ est
contenu dans l'espace sans bord \mbox{$\Int(W)\times B$}.

Le morphisme~$h$ est compact.
Soit en effet une partie compacte~$K$ de~$\gm^n\times B$.
Sa projection dans~$B$ étant compacte, il existe
un disque fermé~$B'$ contenu dans~$B$ tel que $K\subset\gm^n\times B'$.
Alors, $h^{-1}(K)$ est contenu dans~$W\times B'$, donc est compact.

Il résulte enfin du lemme~\ref{lemm.fibre-dim-zero}, 
appliqué au morphisme~$h_t$, 
que la fibre~$h^{-1}(\sq_t(r))$ (égale à $h_t^{-1}(\sq_t(r))$) 
est de dimension nulle.
D'après le lemme~\ref{lemme.finiplat-global}, appliqué maintenant à~$h$,
le morphisme~$h$ est alors fini au-dessus de~$\sq_t(r)$.

L'espace $\gm^n\times B$ est lisse, donc régulier;
l'espace  $W\times B$ est lisse sur l'espace
de Cohen-Macaulay~$W$, donc Cohen-Macaulay lui-même;
ces deux espaces $\gm^n\times B$ et $W\times B$
sont purement de dimension~$n+1$.
Le corollaire~\ref{plat.cohen} assure donc
que~$h$ est fini et plat au-dessus de~$\sq_t(r)$.
Cela conclut la démonstration du lemme~\ref{lemm.finiplat-Delta}.
\end{proof}

Revenons maintenant à la démonstration 
du théorème~\ref{theo.vecteur-volume-transvection-2}.
L'espace $\sq_B(\Delta)$ est une partie connexe et non vide
de~$\gm^n\times B$ (\S\ref{section.shilov}).
Par suite, le morphisme~$h$ possède un degré~$\delta$
au-dessus de $\sq_B(\Delta)$.
Par changement de base, pour tout $t\in B$,
le morphisme~$h_t$ est fini et plat de degré~$\delta$
au-dessus du point~$\sq_t(r)$.

En particulier, le morphisme $h_0=f|_W$ est fini
et plat de degré~$\delta$ au-dessus de~$y=f(x)$.
On a donc $\delta=d$.

Le morphisme $h_1=f'|_W$ est aussi fini
et plat de degré~$\delta$ au-dessus de~$y=f'(x)$.

D'après l'assertion~(3) du \S\ref{imrecisquel-explicite},
les images réciproques de~$S(\gm^n)$ par~$f|_W$ et $f'|_W$
coïncident.
Puisque $f_\trop=f'_\trop$, on en déduit
que pour tout $t\in S(\gm^n)$, $f^{-1}(t)=(f')^{-1}(t)$.
En particulier, la fibre
de~$f$ en~$y$ coïncide avec celle de~$f'$,
si bien que le point~$x$ est le seul antécédent dans~$W$ de~$y$
par~$f'$.  
Il en résulte que $f'$ est fini et plat de degré~$\delta$
en~$x$.
Par définition, on a donc $\deg^x(f') =\delta=\deg^x(f)$,
d'où
le théorème~\ref{theo.vecteur-volume-transvection-2}.
\end{proof}

\begin{lemm}\label{lemm.dimensions-tropicales-fibre}
Soit $p\colon Y\to X$ un morphisme d'espaces analytiques.
Soit $y\in Y$, soit $x=p(y)$;
posons $m=\dim_x(X)$, $n=\dim_y(Y)$ et $d=\dim_y(Y_x)$; supposons $n=m+d$
(c'est par exemple le cas si $p$ est plat en~$y$).
Soit $f\colon X\to \gm^m$ et $g\colon Y\to \gm^d$ des moments;
posons $h=(f\circ p,g)\colon Y \to \gm^n$.

\begin{enumerate}\def\theenumi{\alph{enumi}}\def\labelenumi{\theenumi)}
\item
Pour que $y\in \Sigma_h$, il faut et il suffit que
$y\in\Sigma_{g|_{Y_x}}$ et $x\in \Sigma_f$.

 \item
Supposons $y\in \Int(Y)$.
Pour que $h$ soit de dimension tropicale~$n$ en~$y$,
il faut et il suffit que $g|_{Y_x}$ soit de dimension tropicale~$d$ en~$y$
et que $f$ soit de dimension tropicale~$m$ en~$x$.
 
\item
Supposons $y\in \Sigma_h\cap \Int(Y)$.
Alors le morphisme~$f$ est fini et plat en~$x$ et
les morphismes~$h$ et $g|_{Y_x}$ sont finis et plats en~$y$.

\item
Supposons que $y\in\Sigma_h$ et que $p$ est plat en~$y$.
Le morphisme~$h$ est quasi-fini en~$y$, le morphisme~$f$
est quasi-fini en~$x$ et le morphisme~$g|_{Y_x}$ est quasi-fini en~$y$.
De plus, on a
\[ \deg^y(h) = \deg^x (f) \deg^y (g|_{Y_x}). \]
\end{enumerate}
\end{lemm}
\begin{proof}
\emph a) Les formules explicites décrivant le squelette de~$\gm^n$
montrent  qu'un point~$t$ de~$\gm^n$ appartient à son squelette si et seulement
si sa projection~$s$ dans~$\gm^m$ est squelettique et 
si $t$ appartient au squelette de sa fibre~$\mathbf G^d_{\mathrm m,\mathscr H(s)}$.
L'assertion en résulte.

\emph b)
Comme $y$ n'appartient pas  au bord de~$Y$, il n'appartient pas
au bord de~$Y_x$ et $x$ n'appartient pas au bord de~$X$.
L'assertion~\emph b) du lemme découle alors de
ce qui précède
et de la proposition~\ref{prop.trop.local}, (iv).

\emph c)
On déduit de \emph a) les égalités
$d_k(x)=m$, $d_k(y)=n$ et $d_{\mathscr H(x)}(y)=d$.
D'après le lemme~\ref{lemme.finiplat-local}, \emph c),
le morphisme~$f$ est fini et plat
en~$x$ et   les morphismes~$h$ et $g|_{Y_x}$ sont finis et plats en~$y$,
car $y\in\Int(Y)$.

\emph d)
Considérons le diagramme commutatif
\[
\begin{tikzcd}[sep=large]
 Y \ar{rd}{h} \ar[swap]{d}{g'} \ar[bend right=60]{dd}[swap]{p} \\
 X\times \gm^d  \ar[swap]{r}{f'} \ar{d}{} & \gm^m\times\gm^d  \ar{d} \\
X \ar{r}{f} & \gm^m ,
\end{tikzcd}
\]
où le carré inférieur est cartésien.
Comme $y\in \Sigma_h$, le morphisme~$h$ est quasi-fini et plat en~$y$
(lemme~\ref{lemm.quasi-fini-plat-sq}).
D'après~\emph a), le point $y$ apartient à~$\Sigma_{g|_{Y_x}}$,
donc $g|_{Y_x}$ est quasi-fini et plat en~$y$;
de même, $x\in\Sigma_f$, donc $f$ est quasi-fini et plat en~$x$.

On a $h = f' \circ g'$. 
Comme $h$ est quasi-fini en~$y$, il en est  de même de~$g'$;
par ailleurs, le morphisme~$f'$ est quasi-fini en $g'(y)$,
par changement de base, car $g'(y)$ est au-dessus de~$x$.
Démontrons que $\deg^y(h)=\deg^{g'(y)}(f') \deg^y(g')$;
d'après la proposition~\ref{prop.degre-multiplicatif},
il suffit de démontrer que le morphisme~$g'$ est plat en~$y$.
La projection
$X\times\gm^d\to X$ est plate (\cite{ducros2018}, lemme~4.1.13);
le morphisme~$p$ est plat en~$y$ par hypothèse.
De plus, la fibre de $g'$ en~$x$ s'identifie au morphisme 
$g|_{Y_x}\colon Y_x\to\mathbf G^d_{\mathrm m,\mathscr H(x)}$ 
qui est plat en~$y$.
D'après le critère de platitude par fibres (\cite{ducros2018},
théorème~8.3.8), 
le morphisme~$g'$ est plat en~$y$.

Pour conclure, il suffit alors de démontrer 
les égalités $\deg^{g'(y)}(f')=\deg^x(f)$
et $\deg^y(g')= \deg^y(g|_{Y_x})$.

Commençons par la seconde. Par définition, $\deg^y(g')$ 
est la dimension sur~$\hr {g'(y)}$ de l'anneau local
de $(g')^{-1}(g'(y))$ en~$g'(y)$.
Considérons $X\times \gm^d$ comme $\gm[X]^d$;
alors $g'(y)$ s'identifie au point $g|_{Y_x}(y)$ de $\gm[\hr x]^d$.
De même, la fibre~$(g')^{-1}(g'(y))$ s'identifie à la fibre
$(g|_{Y_x})^{-1}(g|_{Y_x}(y))$. 
Il en résulte l'égalité $\deg^y(g')=\deg^y(g|_{Y_x})$.

Démontrons enfin  que $\deg^{g'(y)}(f')=\deg^x(f)$.
Par définition, $\deg^{g'(y)}(f')$ est la dimension
sur $f'(g'(y))=h(y)$ de l'anneau local 
de $(f')^{-1}(h(y))$ en~$g'(y)$.
Le morphisme~$f'$ est déduit de~$f$ par changement de base. Ainsi,
$(f')^{-1}(h(y))= f^{-1}(f(x)) \otimes_{\hr {f(x)}} \hr {h(y)}$.
L'anneau local de $f^{-1}(f(x))$ en~$x$ est une algèbre
locale artinienne 
de dimension~$\deg^x(f)$ sur~$\hr x$.
Son produit tensoriel par~$\hr {h(y)}$ est le produit des anneaux
locaux de~$(f')^{-1}(h(y))$ en tous les antécédents de~$x$.
Pour conclure la démonstration de l'égalité voulue,
il suffit de démontrer que le point~$x$ n'a qu'un seul antécédent
dans~$(f')^{-1}(h(y))$. 
Cela résulte du fait
que $\hr{f(x)}$ est séparablement clos dans~$\hr{h(y)}$:
le point~$h(y)$ appartient à~$S(\gm[\hr{f(x)}]^n)$,
le corollaire~\ref{coro-skel-ginvariant} s'applique donc. 
\end{proof}


\section{Calibrages canoniques}

Soit $X$ un espace analytique de niveau~$n$.

\begin{prop}\label{prop.vecteurs-volume}
Soit $f\colon X\to T$ un moment et soit $C$ une $f$-cellule de dimension~$n$.
Soit $q$ et $q'$ des morphismes de tores de~$T$
dans~$\gm^n$ tels que $q_\trop$ et $q'_\trop$ soient
injectifs sur~$f_\trop(C)$.
Soit $\sigma_q$ l'unique section affine de~$q_\trop$
dont l'image contient~$f_\trop(C)$; on définit~$\sigma_{q'}$
de façon analogue.

Soit $\mu$ un vecteur-volume sur~$\R^n$.
Pour presque tout $x\in C$,
on a l'égalité \begin{equation}
\label{comparaison-vecteurs-volume} 
\deg^x ({q\circ f})\sigma_{q,*}(\mu) 
=\deg^x ({q'\circ f})\sigma_{q',*}(\mu)
\end{equation}
de vecteurs volumes sur l'espace affine $\langle f_\trop(C)\rangle$.
\end{prop}
\begin{proof}
Notons $C_f=f_\trop(C)$.

\medskip

Supposons d'abord que $q'$ soit de la forme $\phi\circ q$,
où $\phi\colon \gm^n\to\gm^n$ est un morphisme surjectif de tores.
D'après la proposition~\ref{prop.degre-multiplicatif},
on a 
\[ \deg^x(q'\circ f)=\deg^x(\phi\circ q\circ f)
 = \deg^{q\circ f(x)} (\phi) \deg^x(q\circ f). \]
Le morphisme~$\phi$ est fini et plat et 
pour tout $t\in S(\gm^n)$, on a $\deg^t(\phi)=\deg(\phi)$,
en vertu du corollaire~\ref{coro.skel-gmn-funct}.
Par suite, on a $\deg^x(q'\circ f)=\deg(\phi) \deg^x(q\circ f)$
pour tout $x\in C$.

Par ailleurs, $\phi_\trop$ est une application affine bijective
de~$\R^n$ dans lui-même et l'on a $\sigma_{q'}=\sigma_q\circ\phi_\trop^{-1}$.
La valeur absolue du déterminant de l'application linéaire
$\vec\phi_\trop$ est égale au degré de~$\phi$. Par suite,
$\phi_\trop^{-1}$ applique le vecteur volume~$\mu$
sur $\deg(\phi)^{-1}\mu$, de sorte que
$\sigma_{q',*}(\mu)=\deg(\phi)^{-1} \sigma_{q,*}(\mu)$.
La relation voulue est donc vérifiée dans ce cas particulier.

\medskip

Supposons ensuite que $q'_\trop-q_\trop$ soit constante
sur $C_f$.
D'après le théorème~\ref{theo.vecteur-volume-transvection-2},
on a $\deg^x ({q\circ f})=\deg^x (q'\circ f)$ pour tout~$x\in C\cap \Int(X) $;
d'autre part, $C\cap \partial(X)$ est localement négligeable.
De plus, $\sigma_q$ et~$\sigma_{q'}$ diffèrent d'une constante,
si bien que $\sigma_{q,*}(\mu)=\sigma_{q',*}(\mu)$. 
Cela prouve la relation voulue dans ce cas.

\medskip

Montrons maintenant comment la conjonction de ces deux cas permet
d'obtenir le cas général.
Notons $\psi$ l'unique application linéaire de~$\R^n$
dans lui-même telle que $\psi\circ q_\trop$
et $q'_\trop$ diffèrent d'une translation sur~$C_f$.
Compte tenu de la convention  faite au paragraphe~\ref{ss.convention},
 l'espace affine $\langle C_f\rangle$ est $\Q$-rationnel,
si bien que $\psi$ appartient à~$\GL_n(\Q)$.
Il existe donc deux morphismes de tores $\phi$ et $\phi'$
de~$\gm^n$ dans lui-même tels que $\phi'_\trop\circ \psi=\phi_\trop$.
Ainsi, $(\phi'\circ q'\circ f)_\trop$ et  $(\phi\circ q\circ f)_\trop$
diffèrent d'une constante sur~$C$.
Posons $V=f_\trop^{-1}(C_f)$;
c'est un domaine analytique fermé de~$X$
qui contient~$C$.
Sur $V$, on a $(\phi'\circ q'\circ f)_\trop=(\phi\circ q\circ f)_\trop$.

D'après le premier cas particulier, on a 
donc  presque partout les relations
\[ \deg^x ({\phi\circ q\circ f|V}) \sigma_{\phi\circ q,*}(\mu) =
\deg^x ({q\circ f|V}) \sigma_{q,*}(\mu)  \]
et
\[ \deg^x (\phi'\circ q'\circ f|V) \sigma_{\phi'\circ q',*}(\mu) =
\deg^x (q'\circ f|V) \sigma_{q',*}(\mu)  .\]
D'autre part, il résulte du second cas particulier traité que
\[ \deg^x ({\phi\circ q\circ f|V}) \sigma_{\phi\circ q,*}(\mu) =
 \deg^x ({\phi'\circ q'\circ f|V}) \sigma_{\phi'\circ q',*}(\mu) .\]
Comme $V$ est un domaine de~$X$ contenant~$x$, 
on a les égalités de degrés
\[ \deg^x ({q\circ f|V}) = \deg^x ({q\circ f})
\quad\text{et}\quad
 \deg^x ({q'\circ f|V}) = \deg^x ({q'\circ f})
 \]
pour tout~$x\in V$. Cela conclut la démonstration 
de la proposition.
\end{proof}

\subsection{}
Soit~$\Sigma$ une partie squelettique de~$X$,
paracompacte et contenue dans le polyèdre caractéristique~$\Sigma_f$
d'un moment $f\colon V\to T$ défini sur un domaine analytique~$V$ de~$X$.
Nous allons définir un calibrage~$\mu_\Sigma$ de~$\Sigma$.

Soit $\mathscr C$ une décomposition cellulaire 
de~$\Sigma$ en $f$-cellules commodes
(lemme~\ref{lemm.exists-dec-cell-commode}).
Soit~$C$ une cellule de dimension~$n$ appartenant à~$\mathscr C$.
Soit $q\colon T\to \gm^n$ un morphisme de tores 
tel que $q_\trop$ soit injectif sur~$f_\trop(C)$.
D'après la proposition~\ref{prop.commode-fini-plat},
le morphisme $q\circ f$ est fini et plat sur~$\mathring C$
et son degré $x\mapsto \deg^x(q\circ f)$ prend une valeur finie,
strictement positive et constante~$d_C$ sur~$\mathring C$.
Notons $e=\abs{e_1\wedge\dots\wedge e_n}$ le vecteur-volume canonique
de~$\R^n$, considéré comme un calibrage constant,
et posons $\mu_C = d_C (q\circ f)_\trop^*(e)$;
c'est un calibrage standard de la cellule~$C$.
Il ne dépend pas du choix de~$q$ 
d'après la proposition~\ref{prop.vecteurs-volume}.
La famille $(\mu_C)$ définit un calibrage standard de~$\Sigma$;
on le note $\mu_{\Sigma,f,\mathscr C}$.

Vérifions que ce calibrage est indépendant
du choix de~$(V,f,\mathscr C)$.

Soit $V'$ un domaine analytique  de~$X$,
soit $f'\colon V'\to T'$ un moment tel que $\Sigma\subset\Sigma_{f'}$
et soit $\mathscr C'$
une décomposition cellulaire de~$\Sigma$
en $f'$-cellules commodes.
Posons $W=V\cap V'$ et soit $g$ le moment~$(f,f')$ sur~$W$ et
considérons une décomposition cellulaire~$\mathscr D$ de~$\Sigma$
en $g$-cellules commodes
qui est adaptée aux décompositions cellulaires~$\mathscr C$ et~$\mathscr C'$.

Soit $D$ une cellule de dimension~$n$ appartenant à~$\mathscr D$.
Soit $q\colon T\to\gm^n$ un morphisme
de tores tel que $q_\trop$ est injectif sur~$f_\trop(D)$.
Notons aussi $p\colon T\times T'\to T$
la première projection.

Soit $C$ une cellule de~$\mathscr C$ contenant~$D$.
Comme $\mathring D$ ne rencontre pas~$\partial(W)$,
il est contenu dans~$\Int(W/V)$, si bien que 
$q\circ f|_W$ est fini et plat de degré~$d_C$
en tout point de~$\mathring D$.
Il existe aussi un entier~$d_D>0$
tel que  $(q\circ p)\circ g$ soit fini et plat de degré~$d_D$
en  tout point de~$\mathring D$
(proposition~\ref{prop.commode-fini-plat}).
Comme on a  $(q\circ p)\circ g=q\circ f|_W$,
on a $d_C=d_D$. Par ailleurs, $(q\circ g)_\trop^* (e)
= (q\circ f)_\trop^* (e)$.
Par suite, les triplets $(V,f,\mathscr C)$
et $(W,g,\mathscr D)$ induisent le même calibrage sur~$D$.

Cela entraîne que le calibrage~$\mu_{\Sigma,f,\mathscr C}$
de~$\Sigma$
déduit de $(V,f,\mathscr C)$ coïncide avec 
celui déduit de $(W,g,\mathscr D)$;
par symétrie, il coïncide avec celui déduit de~$(V',f',\mathscr C')$.
Notons-le~$\mu_\Sigma$.

\subsection{}\label{sss.calibrage-restriction}
Soit $\Sigma'$ un sous-espace paralinéaire de~$\Sigma$.
Soit $\mathscr C'$ une décomposition cellulaire de~$\Sigma$
en $f$-cellules commodes
qui raffine $\mathscr C\cap\Sigma'$.
La description de~$\mu_{\Sigma',f,\mathscr C'}$
entraîne alors
l'égalité $\mu_\Sigma|_{\Sigma'}=\mu_{\Sigma'}$.

\begin{lemm}\label{lemm.calibrage-canonique}
Soit $X$ un espace analytique de niveau~$n$
et soit~$\Sigma$ une partie squelettique de~$X$.
Il existe un unique calibrage~$\mu_\Sigma$ de~$\Sigma$
vérifiant la propriété suivante: pour tout domaine analytique~$V$
de~$X$, tout moment~$f$ sur~$V$ et tout sous-espace
paralinéaire~$\Sigma'$ de~$\Sigma_f\cap\Sigma$
qui est paracompact, on a $\mu_\Sigma|_{\Sigma'}=\mu_{\Sigma'}$.
Ce calibrage $\mu_\Sigma$ est standard.
\end{lemm}
\begin{proof}
Par définition, l'espace squelettique~$\Sigma$ est G-recouvert
par des sous-espaces paralinéaires~$\Sigma'$
comme dans l'énoncé. Les calibrages formant un G-faisceau,
il existe au plus un tel calibrage.
Pour démontrer son existence, il suffit alors de vérifier
la condition de coïncidence. Or, si $\Sigma'$
est $\Sigma''$ sont des sous-espaces paralinéaires de~$\Sigma$
satisfaisant les conditions de l'énoncé, 
on a $\mu_{\Sigma'}|_{\Sigma'\cap\Sigma''}=\mu_{\Sigma'\cap\Sigma''}
=\mu_{\Sigma''}|_{\Sigma'\cap\Sigma''}$
d'après le paragraphe~\ref{sss.calibrage-restriction}.
\end{proof}

\begin{defi}\label{defi.calibrage-canonique}
Soit $X$ un espace analytique de niveau~$n$, topologiquement séparé,
 et soit $\Sigma$ une partie squelettique de~$X$.
On appelle \emph{calibrage canonique} de~$\Sigma$
le calibrage dont le lemme~\ref{lemm.calibrage-canonique}
affirme l'existence et l'unicité. 
\end{defi}

Par construction, ce calibrage est standard.
Si $\Sigma'\subset\Sigma$, on a $\mu_\Sigma|_{\Sigma'}=\mu_{\Sigma'}$.
Il résulte aussi des constructions
que si $\Sigma$ est contenu dans un domaine analytique~$V$ de~$X$,
le calibrage canonique~$\mu_\Sigma$ de~$\Sigma$,
vu comme partie squelettique de~$V$,
coïncide avec son calibrage canonique comme
partie squelettique de~$X$.

Ce calibrage canonique munit 
l'espace vertébré associé à l'espace analytique~$X$ 
d'une structure d'espace G-tropical calibré de niveau~$n$.
Nous supposerons toujours implicitement que $X$ est muni de cette structure.

\begin{prop}\label{prop.calibrage-chgt-base}
Soit $X$ un espace analytique de niveau~$n$.
Soit $\Sigma$ une  partie squelettique de~$X$,
soit $\mu$ son calibrage canonique.
Soit $L$ une extension complète de~$k$, soit $p\colon X_L\to X$
le morphisme de changement de base, soit $\Sigma_L$
la partie squelettique déduite de~$\Sigma$ par changement de base
(théorème~\ref{theo.chgt-base-polyedre}) 
et soit $\mu_L$ son calibrage canonique.
On a $p_*(\mu_L)=\mu$.
\end{prop}
\begin{proof}
Il suffit de traiter le cas où $X$ est (para)compact
et où il existe un moment $f\colon X\to T$ tel que $\Sigma=\Sigma_f$.
Alors, $\Sigma_L=\Sigma_{f_L}$ (théorème~\ref{theo.chgt-base-polyedre}).
Le morphisme de changement de base~$p$ étant compact, il existe
une décomposition cellulaire~$\mathscr C$ de~$\Sigma$
en $f$-cellules commodes et  une décomposition cellulaire~$\mathscr D$
de~$\Sigma_L$ en~$f_L$-cellules commodes telles que
le couple $(\mathscr D,\mathscr C)$ soit adapté au morphisme~$p$.
Soit $C$ une cellule de dimension~$n$ de~$\mathscr C$.
et soit $q\colon T\to\gm^n$ un morphisme
de tores tel que $q_\trop$ soit injectif sur $f_\trop(C)$.
Soit $x\in \mathring C$ et soit $t=q\circ f(x)\in S(\gm^n)$.
Soit $t_L\in (\gm^n)_L$ l'unique antécédent de~$t$ dans~$S(\gm^n)_L$.
D'après le lemme~\ref{lemm.degre.K/k}, on a 
\[ \deg^x({q\circ f}) = \sum_{  y\in p^{-1}(x) \cap (q\circ f_L)^{-1}(t_L)} \deg^y ({q\circ f_L} ). \]
Les points~$y$ qui apparaissent au second membre de cette égalité
sont des antécédents de~$x$ dans~$\Sigma_{f_L}$. Inversement,
soit $y\in\Sigma_{f_L}$ un point tel que $p(y)=x$; prouvons
que $q\circ f_L(y)=t_L$. Soit $D$ la cellule de~$\mathscr D$ telle
que $y\in\mathring D$; on a $p(D)=C$, donc $(q\circ f_L)_\trop=(q\circ f)_\trop\circ p$  est injectif sur~$D$, si bien que $q\circ f_L$ est de dimension tropicale~$n$ en~$y$. Cela entraîne que $q\circ f_L(y)\in S((\gm^n)_L)$;
comme c'est un antécédent de~$t$, on a $q\circ f_L(y)=t_L$.
On a donc
\[ \deg^x ({q\circ f}) = \sum_{  y\in p^{-1}(x)\cap \Sigma_{f_L}} 
    \deg^y ({q\circ f_L} ). \]
Par conséquent, $\mu_C = \sum_D p_* \mu_D$, où $D$ parcourt l'ensemble
des $n$-cellules~$D$ de~$\mathscr D$ au-dessus de~$C$. Ainsi, $\mu_\Sigma
= p_*(\mu_{\Sigma_L})$.
\end{proof}

\section{Calibrages canoniques et morphismes  finis}

\subsection{}\label{ss.degcal2}
Soit $p\colon Y\to X$ un morphisme 
entre espaces $k$-analytiques de niveau~$n$.
Soit~$\Sigma$ une partie squelettique de~$X$.
Supposons que les fibres de~$p$ au-dessus de tout point de~$\Sigma$ sont 
des ensembles finis.

Soit $x$ un point de~$\Sigma$ et soit $y\in f^{-1}(x)$.
Comme $d_k(x)=n$, 
on a $\dim(f^{-1}(x))=0$, de sorte que $f$ est quasi-finie en~$y$.
Par suite, $f^{-1}(x)$ est un $\hr x$-espace analytique fini.
On a ainsi défini au \S\ref{def.degcal_x} le nombre rationnel $\degcal_x(f)$.

\begin{prop}\label{prop.formule-degre-integral}
Soit $p\colon Y \to X$ un morphisme d'espaces $k$-analytiques de niveau~$n$.
Soit $\Sigma$ une partie squelettique de~$X$, 
soit $\mu_\Sigma$ son calibrage canonique.
Soit $\Tau$ la partie squelettique~$p^{-1}(\Sigma)$ de~$Y$
(proposition~\ref{prop.sigmanf-fini})
et soit $\mu_\Tau$ son calibrage canonique. 

Supposons que le morphisme $p_\Sigma\colon\Tau\to\Sigma$ est 
topologiquement compact.

\begin{enumerate}
\item 
La restriction de la fonction $x\mapsto \degcal_x(p)$ à
toute partie paralinéaire paracompacte de~$\Sigma$
est localement constante en dehors d'une partie paralinéaire
$n$-négligeable.

\item 
On a $p_*\mu_\Tau = \degcal(p) \mu_\Sigma$.
\end{enumerate}
\end{prop}

Sous les hypothèses de~\emph b), les fibres de~$f$ en tout point de~$\Sigma$
sont finies, de sorte que la fonction $x\mapsto\degcal_x(p)$ sur~$\Sigma$
est définie.
De plus, l'assertion~\emph a) donne sens au produit $\degcal(p)\mu_\Sigma$
et montre que c'est un calibrage standard de~$\Sigma$.

\begin{proof}
%
On sait (proposition~\ref{prop.sigmanf-fini})
que $p$ induit une immersion par morceaux compacte de~$\Tau$
sur~$\Sigma$.
Ainsi, l'image directe $p_*\mu_\Tau$ est définie, 
et est un calibrage standard de~$\Sigma$.

 
L'espace $X$ est G-recouvert par des domaines analytiques compacts~$V$
sur lesquels il existe un moment~$f$ tel que 
$\Sigma\cap V\subset\Sigma_f$.
Puisque la formation des calibrages canoniques  et des degrés calibrés
commute à la restriction à un domaine analytique,
on peut remplacer~$X$ par~$V$ et~$Y$ par $p^{-1}(V)$,
ce qui nous ramène à prouver la proposition sous l'hypothèse
supplémentaire que 
$X$ est compact et qu'il existe un moment~$f\colon X\to\gm^n$
tel que $\Sigma$ soit une partie squelettique de~$\Sigma_f$.
En remplaçant~$Y$ par un voisinage analytique compact de~$\Tau$,
on se ramène au cas où $Y$ est compact.

Les moments~$f$ et~$f\circ p$ sont agréables
et l'on a $p^{-1}(\Sigma_f)=\Sigma_{f\circ p}$
et $p^{-1}(\Sigma_f^{(n)})=\Sigma_{f\circ p}^{(n)}$
(prop.~\ref{prop.sigmanf-fini}).
Choisissons une décomposition cellulaire~$\mathscr C$ de $\Sigma_f^{(n)}$
en $f$-cellules commodes,
et une décomposition cellulaire~$\mathscr D$ de $\Sigma_{f\circ p}^{(n)}$ 
en $f\circ p$-cellules commodes,
telles que le couple $(\mathscr D, \mathscr C)$ 
soit adapté à $p_{\Sigma_f^{(n)}}\colon \Sigma_{f\circ p}^{(n)}\to\Sigma_f^{(n)}$, à $\Sigma$ et à~$\Tau$. 

Soit $C$ une $n$-cellule de $\mathscr C$ 
et soit $D_1,\ldots, D_r$ les $n$-cellules de $\mathscr D$ dont l'image
par~$p$ est égale à $C$ ; 
l'immersion par morceaux de $\Sigma_{f\circ p}^{(n)}$
vers $\Sigma_f^{(n)}$ induit au-dessus de $\mathring C$ 
un revêtement dont les feuillets sont les $\mathring D_i$.
Soit $x\in \mathring C$ ; 
pour tout $i$, on note $y_i$ l'unique antécédent de $x$ sur $\mathring D_i$,
et l'on pose~$\xi=f(x)$. 

Par définition d'une cellule commode, le point $x$ appartient à $\Int (X)$.
De même, chacun des $y_i$ appartient $\Int(Y)$.
Le morphisme $f\circ p$ est fini et plat en chacun des $y_i$,
de degré $\deg^{y_i}(f\circ p)$. 
Le morphisme $f$ est fini et plat en $x$, de degré~$\deg^x(f)$.

Pour tout $i$, on a $\deg^{y_i}(f\circ p)=\degcal^{y_i}(f\circ p)$,
car $f\circ p$ est fini et plat en~$y_i$.
Par multiplicativité des degrés
calibrés, on a $\deg^{y_i}(f\circ p)=\degcal^{y_i}(p) \degcal^{x}(f)$,
donc $\deg^{y_i}(f\circ p)=\degcal^{y_i}(p) \deg^{x}(f)$
puisque $f$ est fini et plat en~$x$.
On a donc 
\[ \sum_{i=1}^r \deg^{y_i}(f\circ p)=
\sum_{i=1}^r \degcal^{y_i}(p)\cdot \deg^x(f)
= \degcal_x(p) \deg^x(f). \]
Compte tenu de la définition des calibrages canoniques,
cela entraîne l'égalité voulue.
\end{proof}

\begin{rema}\label{rema.stein-calibrage}
Soit $F$ une extension finie de~$k$ 
et soit $X$ un espace $F$-analytique de niveau~$n$.
On peut aussi voir~$X$ comme un espace $k$-analytique de niveau~$n$; 
pour éviter les confusions nous le noterons ici~$X_0$;
on a $\dim(X_0)=\dim(X)$.

Comme l'extension~$F$ de~$k$ est algébrique,
le squelette canonique
$S(\mathbf G_{\mathrm m,F}^n)$ est l'image réciproque de $S(\gm^n)$
par la projection de $\mathbf G_{\mathrm m,F}^n$ sur~$\gm^n$,
et $S(\gm^n)$ est l'image de $S(\mathbf G_{\mathrm m,F}^n)$
(\cite{ducros2012b}, (0.15)).
On en déduit qu'une partie de~$X$ 
est squelettique si et seulement si
elle l'est en tant que partie de~$X_0$. 

Soit donc $\Sigma$ une partie squelettique de~$X$;
notons $\mu$ son calibrage canonique et~$\mu_0$ celui de~$\Sigma$
vu comme partie squelettique de~$X_0$.
La projection canonique  
de $\mathbf G_{\mathrm m,F}^n$ sur~$\gm^n$ est finie et plate 
de degré~$[F:k]$.
Il découle donc de la définition du calibrage canonique
que $\mu_0=[F:k]\mu$.
\end{rema}

\section{Calibrages et tropicalisations}

\begin{defi}\label{defi.polyedre-tropical}\index{polyèdre tropical}
Soit $n$ un entier, soit $X$ un espace $k$-analytique de niveau~$n$
et soit $f\colon X\to T$ un moment agréable sur~$X$.
On appelle \emph{polyèdre tropical} du moment~$f$
le polyèdre~$f_\trop(\Sigma_f^{(n)})$ muni du calibrage 
image directe du calibrage canonique de $\Sigma_f^{(n)}$.
On le note $\Pi_f$.
\end{defi}

Ce n'est autre que le polyèdre tropical de la carte tropicale
agréable~$f_\trop$.

Lorsque $X$ est compact, rappelons  
(théorème~\ref{theo.proprietes-Sigmaf})
que l'on a l'égalité de polyèdres $\Pi_f=f_\trop(X)^{(n)}$.

Comme le calibrage canonique de~$\Sigma_f^{(n)}$
est standard, il en est de même du calibrage
canonique de~$\Pi_f$.

Comme la formation du calibrage canonique de~$\Sigma_f$
commute à l'extension des scalaires
(proposition~\ref{prop.calibrage-chgt-base}),
il en est de même de la formation du calibrage canonique de~$\Pi_f$.

\subsection{}\label{ss.cal-can-but}
Voyons maintenant comment décrire explicitement le calibrage de~$\Pi_f$.
Soit $\mathscr C$ une décomposition cellulaire de~$\Sigma_f^{(n)}$
vérifiant les propriétés de la proposition~\ref{prop.agreable-decomp}
et soit $\mathscr D=\{f_\trop(C)\;;\; C\in\mathscr C\}$ la décomposition
cellulaire de~$\Pi_f$ qui lui correspond.
Soit $D$ une $n$-cellule appartenant à~$\mathscr D$
et soit $q\colon T\to\gm^n$ un morphisme de tores tel que $q_\trop|_D$
soit injectif.
Soit $\mathscr C_D$ l'ensemble des $n$-cellules $C\in\mathscr C$
telles que $f_\trop(C)=D$; pour $C\in\mathscr C_D$,
soit $d_C$ l'entier naturel que $q\circ f$ soit fini et plat
de degré~$d_C$ en tout  point de~$\mathring C$
(prop.~\ref{prop.commode-fini-plat}).
Notons $d_D=\sum_{C\in\mathscr C_D} d_C$.
Notons $e=\abs{e_1\wedge\dots\wedge e_n}$ le vecteur-volume canonique
de~$\R^n$, considéré comme un calibrage constant.
Soit $\sigma_D\colon \R^n\to T_\trop$ l'unique section affine
de~$q_\trop$ d'image~$\langle D\rangle$.
Par définition du calibrage canonique de~$\Sigma_f^{(n)}$
et de l'image directe d'un calibrage, on a 
$\mu_{\Pi_f,D} = d_D \sigma_{D,*}(e)$.

\begin{prop}
Soit $X$ un espace $k$-analytique de dimension~$n$,
soit $f\colon X\to T$ un moment agréable et soit $x\in\Pi_f$.
Soit $V$ un voisinage de~$x$ dans~$\Pi_f$ qui est une $n$-cellule 
\emph{ouverte} de~$T_\trop$.
Soit $q\colon T\to\gm^n$ un morphisme de tores 
tel que $q_\trop|_V$ est injectif.
Supposons que $q\circ f|_{f_\trop^{-1}(V)}$ 
est fini et plat au-dessus
de~$q_\trop(x)_\sq$, de degré~$d_x$.
Notons $\sigma_q\colon \R^n\to T_\trop$ l'unique section affine
de~$q_\trop$ d'image~$\langle V\rangle$ et
notons $e=\abs{e_1\wedge\dots\wedge e_n}$ le vecteur-volume canonique
de~$\R^n$, considéré comme un calibrage constant.
Alors, le calibrage canonique de~$\Pi_f$ est constant
sur un voisinage de~$x$ dans~$V$,
et est égal à $d_x\sigma_{q,*}(e)$.
\end{prop}
\begin{proof}
Reprenons les notations précédentes.
Comme $D$ est une $n$-cellule qui contient~$x$ et que $V$ est un voisinage
de~$x$ dans~$\Pi_f$, on a $\langle D\rangle = \langle V\rangle$.
En particulier, $q_\trop$ est injective sur~$D$.
Il suffit donc de démontrer que l'on $d_D=d_x$ 
pour toute $n$-cellule~$D\in\mathscr D$ contenant~$x$.

Soit $W$ un voisinage ouvert de~$q_\trop(x)_\sq$  dans~$\gm^n$
tel que $q\circ f$ induise un morphisme fini et plat
de $f_\trop^{-1}(V) \cap (q\circ f)^{-1}(W)$ sur~$W$, de degré~$d_x$.
Soit~$\Delta$ une cellule contenue dans~$D$, 
contenant~$x$ et de dimension~$n$,
telle que $q_\trop(\Delta)_\sq\subset W$.

Démontrons que l'on a $d_x=\sum d_C$.
Soit $y$ un point de~$\mathring \Delta\cap W$.
Démontrons que les antécédents de $q_\trop(y)_\sq$ 
sont les antécédents de~$y$ dans~$\Sigma_f^{(n)}$ par $f_\trop$.

Soit $z\in f_\trop^{-1}(V)$ tel que $q\circ f(z)=q_\trop(y)_\sq$.
Comme $\sq\circ q_\trop$ est injective sur~$V$, on a $f_\trop(z)=y$.
Prouvons alors que $z\in \Sigma_f^{(n)}$.
Par définition, on a $z\in \Sigma_f$.
Comme $q\circ f|_{f^{-1}(V)}$ est fini et plat en~$z$, 
l'application $q\circ f$ est ouverte au voisinage de~$z$, 
donc l'image de tout voisinage de~$z$ par $q\circ f$
contient un voisinage  de~$q\circ f(z)=q_\trop(y)_\sq$,
donc $q\circ f$ est de dimension tropicale~$n$ en~$z$.
Cela entraîne que $f$ est de dimension tropicale~$n$ en~$z$,
d'où $z\in\Sigma_f^{(n)}$.

Soit $z\in \Sigma_f^{(n)}$ tel que $f_\trop(z)=y$.
Soit $C$ l'unique cellule de~$\mathscr C$ contenant~$z$;
alors $f_\trop $ induit un homéomorphisme de~$C$ sur~$D$,
si bien que $(q\circ f)_\trop$ induit un homéomorphisme de~$C$
sur~$q_\trop(D)$;  il en résulte que $q\circ f$
est de dimension tropicale~$n$ en tout point de~$C$.
Par suite, $q\circ f(z)\in S(\gm^n)$,
d'où $q\circ f(z) = q_\trop(y)_\sq$ puisque $\trop((q\circ f)(z))
= q_\trop(f_\trop(z))=q_\trop(y)$.
\end{proof}

\begin{rema}
Soit $V$ une partie ouverte de~$\Pi_f$ dont l'adhérence est une cellule
de dimension~$n$ et telle que $f_\trop^{-1}(V)$ soit sans bord.
Soit $q$ un morphisme de tores $q\colon T\to\gm^n$ tel
que $q_\trop|_V$ soit injective.
Supposons que $q\circ f$ soit compact au-dessus d'un voisinage
de~$q_\trop(V)_\sq$.
Cette hypothèse est automatiquement vérifiée lorsque $X$ est compact.
Elle l'est aussi dans le cas algébrique: comme $\dim(X)\leq n$, 
il existe un ouvert de Zariski algébrique~$U$ de~$\gm^n$
non vide, tel que $q\circ f$ soit fini et plat au-dessus de~$U^\an$,
et $U^\an$ contient~$S(\gm^n)$ (lemme~\ref{lemm.squelette-canonique}).

Sous ces conditions, il découle du lemme~\ref{lemme.finiplat-global}
que le morphisme $q\circ f|_{f_\trop^{-1}(V)}$
est fini et plat au-dessus de tout point 
$x\in q_\trop(V)_\sq$; notons~$d_x$ son degré.
Comme $q_\trop(V)_\sq$ est connexe et non vide, il existe un unique entier~$d$
tel que $d_x=d$ pour tout $x\in q_\trop(V)_\sq$.
D'après la proposition précédente, le calibrage $\mu_{\Pi_f}$
est donc constant, égal à $d\sigma_{q,*}(e)$, sur~$V$
(où $\sigma_q$ et $e$ sont comme dans cette proposition).
Remarquons enfin que $d$ est strictement positif.
\end{rema}

\section{Cônes tropicaux}

\subsection{}
Dans les applications arithmétiques de la théorie 
des formes et courants sur un espace de Berkovich, 
nous aurons souvent à travailler avec des situations analytiques
issues de modèles formels. Nous  chercherons alors à décrire
des objets analytiques (tel l'opérateur de Monge-Ampère)
en termes de théorie de l'intersection sur la fibre spéciale.
Pour cela, nous ferons de la géométrie analytique sur cette fibre
spéciale, considérée comme espace analytique sur le corps
\emph{trivialement valué}~$\tilde k$.

Le résultat crucial sur lequel nous nous appuierons pour relier géométrie
analytique sur $k$ et géométrie analytique sur $\widetilde k$ (ce dernier étant trivialement valué) 
est le suivant ; rappelons que le cône tropical d'un moment $f$ en un point de $\Sigma_f^{(n)}$ 
a été défini en \ref{defi.cone-tropical}.
et son calibrage en~\ref{defi.cone-tropical-calibre}.

\begin{theo}\label{theo.comparaison-cones}
Soit $k$ un corps ultramétrique 
complet algébriquement clos.
Soit $(X,x)$ un germe d'espace analytique sans bord, génériquement quasi-lisse,
purement de dimension~$n$; on suppose $d_k(x)=n$.
Soit $f_1,\dots,f_m$ des fonctions analytiques sur~$(X,x)$ telles
que $\abs{f_i(x)}=1$ pour tout~$i$; soit $f\colon (X,x)\to\gm^n$ le moment
qu'elles définissent.
Soit $V$ une $\tilde k^1$-variété propre de corps des fonctions $\hrt x^1$
et soit $\xi$ le point de~$V^\an$ 
correspondant à la valeur absolue triviale sur~$\widetilde k^1(V)$;
soit $\tilde f\colon (V^\an,\xi)\to \mathbf G^m_{\mathrm m,\tilde k}$
le moment définit par ~$\widetilde{f_1(x)},\dots,\widetilde{f_m(x)}$.

Les dimensions tropicales de~$f$ en~$x$ et de~$\tilde f$ en~$\xi$ coïncident.
Si elles sont toutes deux égales à~$n$,
les cônes tropicaux (calibrés)~$\Pi_{\tilde f,\xi}$ et $\Pi_{f,x}$  
dans~$\R^m$ sont égaux.
\end{theo}
Nous faisons observer au  lecteur que ce théorème 
vaut si la valuation de~$k$ est triviale.  Sa preuve consistera
à montrer que les deux cônes tropicaux dont on veut prouver l'égalité
peuvent être décrits exactement de la même façon.
Pour, il sera \emph{essentiel} de ne pas avoir à supposer que
la valuation de~$k$ est non triviale.

\begin{proof}
Comme $X$ est sans bord, 
la dimension tropicale de~$f$  en~$x$ est égale au degré de transcendance
résiduel gradué de $f_1(x),\dots,f_m(x)$, c'est-à-dire 
à celui de l'extension $\widetilde k(\widetilde f_1(x),\dots,\widetilde f_m(x))$ 
de~$\widetilde k$. Mais comme les $\widetilde{f_i(x)}$ sont tous de degré $1$, 
ce degré de transcendance est aussi celui 
$\widetilde k^1(\widetilde f_1(x),\dots,\widetilde f_m(x))$
sur $\widetilde k^1$. 
Calculons maintenant la dimension tropicale de
de~$\widetilde f$ en~$\xi$. Le corps complet $\hr \xi$ est
le corps trivialement valué $\widetilde k^1(\xi)=\hrt x^1$, si bien
que $\hrt \xi=\hrt x^1$ et que $\widetilde {g(\xi)}=g$
pour tout $g\in \widetilde k^1(V)=\hrt x^1$. 
Puisque $V^\an$ est sans bord, il s'ensuit que
la dimension tropicale de
de~$\widetilde f$ en~$\xi$
est égale 
au degré de transcendance 
de 
$\widetilde k^1(\tilde f_1(\xi),\dots, \tilde f_m(\xi))$
sur $\widetilde k^1$. 
Cela démontre que les deux dimensions tropicales considérées coïncident.
 
Supposons les toutes égales à~$n$, de sorte que les cônes tropicaux
$\Pi_{f,x}$ et $\Pi_{\tilde f,\xi}$ sont définis.
La coïncidence des cônes rationnels sous-jacents 
est alors une conséquence de la remarque 
\ref{rem-description-trop-germ}. 
Il reste à démontrer que 
leurs calibrages coïncident. À partir de maintenant, nous n'allons manipuler dans cette preuve 
que les réductions ordinaires, c'est-à-dire les parties de degré $1$ des réductions graduées. 
Pour alléger les notation nous écrirons donc (uniquement dans la suite de la preuve) 
$\hrt x$ au lieu de $\hrt x^1$, etc. 

Commençons par rappeler la définition des calibrages sur nos cônes tropicaux. 
Notons tout d'abord que $f_\trop(x)=0$.
On choisit une application linéaire de~$\R^m$
sur $\R^n$, définie sur $\Z$, dont la restriction 
 à chaque face maximale de~$\Pi_{f,x}$ est injective. 
On note~$g$ la composée du moment~$f$
et du morphisme de tores $q\colon \gm^m\to \gm^n$ 
tel que $q_\trop$ soit l'application linéaire choisie. 
Posons $g=(g_1,\dots,g_n)$; on a $g_\trop(x)=0$.
Pour tout~$i$, $g_i$ est un monôme (sans coefficient) en $f_1,\dots,f_m$,
donc $\abs{g_i(x)}=1$. 
De plus, la dimension tropicale de~$g$ en~$x$ est égale à~$n$.
Par suite, $g(x)$ est le point de Gauß $\eta_1$ de $\gm^n$.

Par réduction, le morphisme~$g$ fournit un morphisme de germes
$\tilde g \colon (V^\an, \xi)\to (\mathbf G^n_{\mathrm m,\tilde k}, \eta_{1})$. 
Notons que $\tilde g=\tilde q\circ \tilde f$, 
où $\tilde q\colon \mathbf G^m_{\mathrm m,\tilde k} \to
\mathbf G^n_{\mathrm m,\tilde k}$ est le morphisme de tores donné par
les mêmes monômes que~$q$.

D'après le lemme~\ref{lemme.finiplat-local}, le point~$x$ est isolé
dans~$g^{-1}(\eta_1)$ et le morphisme~$g$
est fini et plat en~$x$.  

Comme $d_k(x)=n$, le groupe de la valuation de $\hr x$  est 
de torsion modulo~$\abs{k^\times}$, donc égal à $\abs{k^\times}$
puisque celui-ci est divisible, $k$ étant algébriquement clos.
Soit~$F$ la sous-extension non ramifiée maximale de $\hr x/\hr {\eta_1}$;
son corps résiduel~$\tilde F$ est la fermeture séparable de $\hrt {\eta_1}$ 
dans $\hrt x$. 
Notons que $\hrt {\eta_1}$ est égal à $ \tilde k(T_1,\ldots, T_n)$. 
Les corps $\hr {\eta_1}$, $\hr x$ et~$F$ sont stables 
(lemme \ref{lemm.abhyankar-finitude}). 
On a donc 
\begin{gather}
P = [\hrt x: \tilde F] =[\hr x: F] \\
D = [\tilde F: \tilde k(T_1,\ldots, T_n)]=[F: \hr{\eta_1}]  \\
DP = [\hrt x: \tilde k(T_1,\ldots, T_n)]=[\hr x: \hr{\eta_1}].
\end{gather}
Remarquons que $P$ est une puissance de l'exposant caractéristique résiduel
et que $\hrt x^P\subset \tilde F$.

Comme l'extension~$F$ de~$\hr{\eta_1}$ est séparable,
le lemme de Krasner assure qu'il existe une
extension finie séparable $F_1$ de~$k(T_1,\dots,T_n)$, de degré~$D$,
induisant~$F$ par produit tensoriel avec $\hr{\eta_1}$.
Ce corps~$F_1$  est le corps des fonctions d'un $k$-schéma
de type fini intègre~$\mathscr Y$ 
muni d'un morphisme étale dominant vers~$\gm^n$, génériquement de degré~$D$.
Par construction, $\eta_1$ possède un unique antécédent~$y\in\mathscr Y^\an$
et l'on a $\hr y=F$. Posons $Y=\mathscr Y^\an$.

Par construction, le $\hr {\eta_1}$-plongement $\hr y\to \hr x$ 
fournit un morphisme $\hr y\otimes_{\hr{\eta_1}}\hr x \to \hr x$,
c'est-à-dire un point $\hr x$-rationnel de $X\times_{\gm^n} Y$
au-dessus de~$x$. D'après l'antiéquivalence de catégories
entre espaces étales sur~$(X,x)$ et algèbres étales sur~$\hr x$
(\cite{berkovich1993}, théorème~3.4.1), le morphisme de germes
$g\colon (X,x)\to (\gm^n,\eta_1)$ se relève donc en un morphisme
$h\colon (X,x)\to (Y,y)$ tel que $h(x)=y$.

\begin{lemm}
Le morphisme~$h$ est fini et plat de degré~$P$.
\end{lemm}
\begin{proof}
Comme $\mathscr Y$ est intègre et étale sur~$\gm^n$, il est 
réduit et purement de dimension~$n$; il en est donc de même de~$Y$.
On a $d_k(y)=d_k(\eta_1)=n$. Comme $X$ est sans bord,
le lemme
\ref{lemme.finiplat-local}
entraîne que le morphisme de germes~$h$
est fini et plat.
Comme $X$ est génériquement lisse et comme $d_k(x)=n$,
l'espace~$X$ est lisse au voisinage de~$x$; en particulier,
l'anneau local $\mathscr O_{X,x}$ est un corps.
Le théorème~6.3.7 de~\cite{ducros2018} entraîne
alors que le morphisme d'anneaux locaux 
$\mathscr O_{X,x}\to \mathscr O_{X_y,x}$
est régulier.
En conséquence, la fibre
$h^{-1}(y)$ est réduite en~$x$.
Le degré du morphisme~$h$ en~$x$ est alors égal au degré
de l'extension~$\mathscr H(x)/\mathscr H(y)$, c'est-à-dire~$P$.
\end{proof}

Soit $\Omega$ un voisinage affinoïde de~$x$ dans~$X$
tel que le polyèdre calibré $\Pi_{f|_\Omega}$ 
soit un voisinage de~$0$, étoilé en $0$,
dans le cône calibré~$\Pi_{f,x}$ ; 
on exige en outre que $x$ soit le seul antécédent de~$\eta_1$
dans~$\Omega$. 
Le morphisme~$g|_\Omega$ est alors fini et plat au-dessus de~$\eta_1$.
Soit $W$ un voisinage affinoïde de~$\eta_1$ 
dans $\gm$ 
tel que $g$ induise un morphisme fini et plat 
surjectif de
$\Omega\cap g^{-1}(W)$  dans~$W$.
Il existe une cellule~$C$ de~$\R^n$, voisinage de l'origine,
telle que $C_\sq\subset W$.
Soit $W'=\trop^{-1}(C)$; c'est un domaine affinoïde de~$\gm^n$.
Alors $g_\trop(\Omega\cap g^{-1}(W\cap W'))=C$
et $g(\Omega\cap g^{-1}(W\cap W'))$ contient~$C_\sq$.
Quitte à remplacer~$\Omega$ par $\Omega\cap g^{-1}(W\cap W')$,
on peut donc supposer
que $g_\trop(\Omega)$ est une cellule, voisinage de l'origine
dans~$\R^n$, et que $g|_\Omega$ est fini et plat 
au-dessus de~$g_\trop(\Omega)_\sq$.


Soit $\mathscr C$ une décomposition cellulaire du cône tropical~$\Pi_{f,x}$,
dont les cellules sont des cônes et qui est adaptée à son calibrage;
on suppose en outre que son image par~$q_\trop$ 
se prolonge en une décomposition cellulaire de~$\R^n$.


Soit $C$ une $n$-cellule ouverte de~$\Pi_{f,x}$;
posons $C'=g_\trop (C)$.
Notons aussi $D=C\cap f_\trop(\Omega)$ et $D'=q_\trop(D)$.
L'ensemble~$D$ est ouvert et fermé dans 
$q_\trop^{-1}(D')\cap f_\trop(\Omega)$,
donc $f_\trop^{-1}(D)\cap\Omega$ 
est ouvert et fermé dans $g_\trop^{-1}(D')\cap\Omega$.
Par suite, la restriction de~$g$ à $f_\trop^{-1}(D)\cap\Omega$ 
est finie et plate au-dessus de~$D'_\sq$.
Comme $D'_\sq$ est connexe et non vide, 
il existe un unique entier~$N_{C,\Omega}$ 
tel que cette restriction soit de degré~$N_{C,\Omega}$.
Le vecteur-volume~$\mu_C$ correspondant au calibrage
de la cellule~$C$ de~$\Pi_{f,x}$ est égal à $N_{C,\Omega} \sigma_C(e)$,
où $e$ est le vecteur-volume canonique de~$\R^n$
et $\sigma_C\colon\R^n\to\R^m$ est l'unique section de~$q_\trop$
d'image~$\langle C\rangle$. Cela prouve que l'entier~$N_{C,\Omega}$
ne dépend pas du choix du voisinage~$\Omega$ choisi comme ci-dessus;
on le notera~$N_C$.

Le cône tropical~$\Pi_{\tilde f,\xi}$ jouit d'une description similaire,
à l'aide d'un voisinage affinoïde~$\Upsilon$ de~$\xi$ dans~$V^\an$
et de décompositions cellulaires de~$\tilde f(\Upsilon)$,
de $\Pi_{\tilde f,\xi}$ et de $\R^n$ adaptées à~$q_\trop$.
Par raffinement,
on suppose que ces décompositions cellulaires coïncident
avec celles introduites précédemment pour~$\Pi_{f,x}$.

Pour toute $n$-cellule~$C$ de~$\Pi_{\tilde f,\xi}$ qui rencontre
la $n$-cellule ouverte~$D$ de~$\mathscr C$, d'image~$D'$ dans~$\R^n$,
la restriction de~$\tilde g$ à $f_\trop^{-1}(C)\cap \Upsilon $
est finie  et plate au-dessus de~$D'_\sq$; on note~$\tilde N_C$
le degré correspondant, de sorte que le vecteur-volume~$\tilde\mu_C$
donné par le calibrage de~$\Pi_{\tilde f,\xi}$ est égal
à $\tilde N_C\sigma _C(e)$.

D'après la proposition suivante, on a 
$N_C=\tilde N_C$
pour toute $n$-cellule~$C$ de $\Pi_{f,x}$.
Le théorème s'ensuit.
\end{proof}

%
 
\begin{prop}\label{prop.NC-NCtilde}
Pour toute $n$-cellule~$C$ de $\Pi_{f,x}$,
on a $N_C=\tilde N_C$.
\end{prop}
\begin{proof}
On peut définir la $n$-cellule conique~$C$ 
comme le lieu de validité dans~$\R^m$
d'une conjonction d'inégalités $\Phi=\bigwedge \{\sum Pa_{r,i}x_i <0\}$.
De même, le cône~$C'$ de~$\R^n$ 
est défini par une conjonction 
d'inégalités~$\Phi'=\bigwedge \{\sum a'_{s,j}y_j<0\}$.

On note $\Psi$ la conjonction des inégalités monomiales
$\prod \abs{f_i}^{Pa_{r,i}}<1$; elle définit un ouvert de~$X$.
On note aussi $\widetilde\Psi$ la conjonction des inégalités monomiales
$\prod \abs{\widetilde{f_i(x)}}^{Pa_{r,i}}<1$;
en interprétant $\widetilde{f_i(x)}$ comme une fonction rationnelle
sur~$V$, cette conjection définit un ouvert de~$V^\an$.

De même, notons $\Psi'$ la  conjonction des inégalités
$\prod \abs{T_j}^{a'_{s,j}}<1$, définissant un ouvert
de $\mathbf G^n_{\mathrm m,k}$.
On notera  $\widetilde{\Psi'}$ la 
même conjonction d'inégalités
définissant un ouvert de $\mathbf G^n_{\mathrm m,\tilde k}$.

Le degré $N_C$ peut être caractérisé ainsi : 
dans l'ouvert de~$\Omega$ défini par la condition~$\Psi$,
tout point de $S(\gm^n)$ 
satisfaisant $\Psi'$ et suffisamment proche de $\eta_{1}$ a $N_C$ 
antécédents  (comptés avec degrés de l'extension de corps résiduels
et multiplicité) par  le morphisme~$g$.

Pour tout~$i$, choisissons une fonction~$h_i\in \mathscr O_{Y,y}$
telle que $\abs{h_i(y)}=1$ et $\widetilde{h_i(y)}=\widetilde{f_i(x)}^P$;
c'est possible puisque $\hrt x^P\subset \tilde F=\hrt y$
et puisque, $Y=\mathscr Y^\an$ étant sans bord (donc presque bon en
tout point), $\kappa(y)$ est dense dans~$\hr y$.
Au voisinage de~$x$, 
les fonctions~$f_i^P$ et~$h_i\circ h$
ont donc même valeur absolue.
Notons $\Theta$ la conjonction des inégalités
$\prod \abs{h_i}^{a_{r,i}}<1$. Elle définit un ouvert de~$Y$
dont l'image réciproque par~$h$ se décrit au voisinage, 
de $x$ par la conjonction des inégalités monomiales
$\prod \abs{f_i}^{Pa_{r,i}}<1$, c'est-à-dire par $\Psi$. 

L'extension~$\tilde F$ définit un 
revêtement fini étale 
$\widetilde{\mathscr Y}$ de $(\gm ^n)_{\tilde k}$ au voisinage 
de son point générique et
le morphisme $\tilde g\colon (V^\an, \xi)\to (\gm^n)_{\tilde k}$ se factorise 
par un morphisme 
$\tilde p\colon \widetilde{\mathscr Y^\an}\to (\gm^n)_{\tilde k}$. 
Soit $\zeta=\tilde p(\xi)$; alors le morphisme
$\tilde h\colon (V^\an,\xi)\to (\widetilde{\mathscr Y}^\an, \zeta)$
est fini et plat de degré~$P$.  De plus, le morphisme
$\tilde p$ est fini et étale en~$\zeta$.
La condition $\widetilde{\Psi}$ ne portant, 
par construction,
que sur des normes de fonctions sur 
$\widetilde{\mathscr Y}$,
elle définit une condition~$\widetilde{\Theta}$
sur~$\widetilde{\mathscr Y}^\an$.

\begin{lemm}\label{lemm.antecedents}
\begin{enumerate}
\item
Par le morphisme~$g$, 
tout point de~$S(\gm^n)$ satisfaisant~$\Psi'$ 
et suffisamment proche de~$\eta_1$ a exactement~$N_C/P$ 
antécédents (comptés avec degrés de l'extension de corps résiduels)
dans~$Y$ satisfaisant~$\Theta$.
\item
Par le morphisme~$\tilde g$,
tout point de $S(\mathbf G_{\mathrm m,\tilde k}^n)$ 
satisfaisant $\widetilde{\Psi'}$ et suffisamment 
proche de $\eta_{1,\tilde k}$ a $\tilde N_C/P$ 
antécédents (comptés avec degrés de l'extension de corps résiduels)
dans~$\widetilde{\mathscr Y}^\an$  satisfaisant~$\widetilde{\Theta}$.
\end{enumerate}
\end{lemm}
\begin{proof}
Nous ne démontrons que la première assertion,
la seconde en est le cas particulier d'un corps trivialement valué.

On commence par choisir des voisinages affinoïdes~$X_1$ de~$x$ 
dans~$\Omega$, $Y_1$ de~$y$ dans~$Y$ et $Z_1$ de~$\eta_1$ dans~$\gm^n$,
tels que $h$ induise un morphisme fini et plat de degré~$P$,
$h_1\colon X_1\to Y_1$, 
et $p\colon \mathscr Y\to \mathscr G_{\mathrm m}^n$ 
induise un morphisme fini étale
$p_1\colon Y_1\to \gm^n$. C'est possible car le morphisme~$h$
est fini et plat en~$x$ et le morphisme~$p$ est fini et étale
en~$y$.

Comme $x$ est le seul antécédent de~$\eta_1$ par~$g$ dans~$\Omega$,
le voisinage~$X_1$ de~$x$ est un voisinage de la fibre de~$g$; 
il contient donc l'image inverse d'un voisinage de~$\eta_1$.
Autrement dit, tous les antécédents dans~$\Omega$ d'un point assez proche
de~$\eta_1$ appartiennent à~$X_1$.
De même, tous les antécédents dans~$Y$ d'un point assez proche de~$\eta_1$
appartiennent à~$Y_1$.

Soit $z$ un point de~$\gm^n$ assez proche de~$\eta_1$ satisfaisant
la condition~$\Psi'$. 
Notons $y_1,\dots,y_a$ ses antécédents dans~$Y_1$ par~$p$.
Pour $i\in\{1,\dots,a\}$, notons $x_{ij}$ les antécédents de~$y_i$ dans~$X_1$
par~$h$; on remarque que $x_{ij}$ vérifie la condition~$\Psi$
si et seulement si~$y_i$ vérifie la condition~$\Theta$.
Par suite, la définition de~$N_C$ entraîne que
\[ N_C = \sum_{\Psi(x_{ij})} 
\rang_{\mathscr O_{Z,\eta_1}} (\mathscr O_{X,x_{ij}}) 
= \sum_{\Theta(y_i)} 
\rang_{\mathscr O_{Z,\eta_1}} (\mathscr O_{Y,y_i})
   \rang_{\mathscr O_{Y,y_i}}(\mathscr O_{X,x_{ij}}) . \]
Puisque $h_1$ est fini et plat de degré~$P$, on a 
\[ P = \sum_{j} \rang_{\mathscr O_{Y,y_i}} (\mathscr O_{X,x_{ij}}) \]
pour tout~$i\in\{1,\dots,a\}$.
Puisque $p_1$ est étale, on a aussi 
\[ 
\rang_{\mathscr O_{Z,\eta_1}} (\mathscr O_{Y,y_i})
= [ \hr{y_i}:\hr {\eta_1}]. \]
Ainsi,
\[ N_C = P \sum_{\Theta(y_i)} [\hr{y_1}:\hr{\eta_1}], \]
comme il fallait démontrer.
\end{proof}
%
%
 
Soit $\gimel$ le groupe~$\R^\times_+\times\R$ muni de l'ordre lexicographique;
son élément neutre est $(1,0)$ et les éléments de la forme $(1,x)$ avec
$x>0$ vérifient $(1,0) <(1,x) < (t,0)$ pour tout $t\in\R_+^\times$ 
tel que $t>1$.
Choisissons un élément $\mathbf g$ de~$\gimel^n$ 
de la forme $(1,c)$, où $c\in\R^n$ appartient à la cellule ouverte~$C$.
Il définit une valuation de Gauß $\eta_{\mathbf g}$  sur $k(T_1,\dots,T_n)$
et une valuation de Gauß $\eta_{\mathbf g,\tilde k}$  sur $\tilde k(T_1,\dots,T_n)$.

Appliquons l'« élimination des quantificateurs » dans~\textsc{acvf}
(\cite{robinson1956}, voir aussi~\cite{weispfenning1984}) :
soit $\Lambda_C$ une formule sans quantificateur du langage des 
anneaux  valués, 
à paramètres dans~$k$
qui est équivalente dans la théorie~\textsc{acvf} à la formule:
\[ (x\in \gm^n) \wedge \Psi'(x) \wedge
 \text{« $x$ a $N_C/P$ antécédents dans $\mathscr Y$ satisfaisant~$\Theta$ »}.\]
Compte tenu de la formule définissant les valuations de Gauß~$\eta_\gamma$,
on voit
 qu'il existe une formule~$\Lambda'_C$ sans quantificateur
du langage des groupes ordonnés à paramètres dans~$\abs{k^\times}$
tel que pour tout modèle~$\Gamma$ de la théorie~\textsc{doag}
contenant~$\abs{k^\times}$ et tout~$\gamma\in\Gamma$, l'élément~$\gamma$ 
satisfait~$\Lambda'_C$ si et seulement si le point~$\eta_\gamma$ 
satisfait~$\Lambda_C$. 

Le lemme \ref{lemm.antecedents}
ci-dessus
assure l'existence d'un nombre réel~$\eps>0$
tel que pour tout $r\in(\R_+^*)^n$, le point~$\eta_r\in\gm^n$ 
satisfait~$\Lambda_C$ 
dès que $\eps<\abs{r_i}<1/\eps$ pour tout~$i$ et que $r$ satisfait~$\Psi'$.

Lorsque la valeur absolue de~$k$ n'est pas triviale, le groupe $\abs{k^\times}$
est dense dans~$\R_+^\times$, car $k$ est algébriquement clos,
et l'on peut supposer que $\eps$ appartient à $\abs{k^\times}$.
L'implication
\[ \forall i (\eps<\abs{r_i}<1/\eps) \wedge \Psi'(r)
\Rightarrow \Lambda'_C(r) \]
est vérifiée dans le modèle $\R_+^*$ de la théorie~\textsc{doag};
elle est donc vérifiée dans le modèle~$\gimel$.
Par suite, la valuation~$\eta_{\mathbf g}$ satisfait~$\Lambda_C$.

Dans le cas où la valeur absolue de~$k$ est triviale,
on remarque que si $\eta_r$ satisfait ces
conditions, alors pour tout~$\alpha>0$,
le point~$\eta_{r^\alpha}$ les satisfait également,
car ces points définissent la même valuation; on voit ainsi que l'on
peut omettre la condition $\eps<\abs{r_i}<1/\eps$ dans l'argument
ci-dessus et l'on en déduit que 
la valuation~$\eta_{\mathbf g}$ satisfait~$\Lambda_C$.

Retraduisons ceci en termes de valuations.
La valuation~$\eta_{\mathbf g}$ du corps $k(T_1,\dots,T_n)$
possède ainsi $N_C/P$ extensions
à $k(\mathscr Y)$
et satisfaisant $\Theta$, comptées avec multiplicités. 
A priori, la  multiplicité associée à une telle extension 
est le degré de son hensélisé
sur celui du hensélisé de $(k(T_1,\dots,T_n),\eta_1)$;
par stabilité, cette multiplicité est simplement égale au produit
des indices d'inertie et de ramification de l'extension. 

On prouve de même que 
la valuation~$\eta_{\mathbf g,\tilde k}$
possède $\tilde {N_C}/P$ extensions à $\widetilde k(\widetilde{\mathscr Y})$
satisfaisant $\widetilde Theta$, 
comptés avec multiplicités ; et la multiplicité associée à une telle extension 
est simplement égale au produit
de ses indices d'inertie et de ramification. 

Il résulte alors du lemme suivant que $N_C/P=\widetilde{N_C}/P$,
si bien que $N_C=\widetilde N_C$, ce qui conclut
la démonstration de la proposition~\ref{prop.NC-NCtilde}.
\end{proof}

\begin{lemm}
Par composition avec la valuation associée à~$y$,
toute valuation~$\tilde v$ sur~$\tilde k(\tilde{\mathscr Y})$
triviale sur~$\tilde k$ induit une valuation~$v$ sur $k(\mathscr Y)$.

\begin{enumerate}
\item
Cette application $\tilde v\mapsto v$ induit
une bijection de l'ensemble des prolongements 
à~$\tilde k(\tilde{\mathscr Y})$ de~$\eta_{\mathbf g,\tilde k}$
sur l'ensemble des prolongements 
à~$k(\mathscr Y)$ de~$\eta_{\mathbf g}$.

Soit $\tilde v$  et~$v$ deux telles valuations.

\item
Les indices d'inertie de~$v$ et de~$\tilde v$ sont égaux.

\item
Les indices de ramification de~$v$ et de~$\tilde v$ sont égaux.

\item
La valuation~$v$ vérifie la condition~$\Theta$
si et seulement si la valuation~$\tilde v$ vérifie
la condition~$\tilde\Theta$.
\end{enumerate}
\end{lemm}
\begin{proof}
La valuation~$\eta_{\mathbf g}$ raffine~$\eta_1$ qui n'a qu'un seul
prolongement à~$k(\mathscr Y)$. En effet, comme $k(\mathscr Y)$ est finie
sur $k(T_1,\dots,T_n)$, le groupe des valeurs d'une valuation 
de~$k(\mathscr Y)$ qui prolonge~$\eta_1$ contient celui de~$\hr{\eta_1}$
comme sous-groupe d'indice fini, donc coïncide avec celui-ci puisqu'il
est égal à~$\abs{k^\times}$ et que $k$ est algébriquement clos.
En particulier, tout prolongement de~$\eta_1$ à~$k(\mathscr Y)$ 
appartient à~$Y^\an$, donc est égal à~$y$ 
qui est le seul antécédent de~$\eta_1$ dans~$Y^\an$.
La première assertion résulte alors de ce que
la valuation~$\eta_{\mathbf g}$
est la composée de la valuation de Gauß usuelle~$\eta_1$
sur~$k(T_1,\dots,T_n)$ et de la valuation $\eta_{\mathbf g,\tilde k}$
sur $\tilde k(T_1,\dots,T_n)$,

Puisque les extensions résiduelles correspondantes coïncident, 
les indices d'inertie de~$v$ et de~$\tilde v$ sont égaux.
Puisque $\abs{\hr y^\times}=\abs{\hr{\eta_1}\times}=\abs{k^\times}$,
leurs indices de ramification sont aussi égaux.

Démontrons maintenant que $v$ vérifie la condition~$\Theta$
si et seulement si $\tilde v$ la condition~$\tilde\Theta$.
Par construction, $\Theta$ est une conjonction d'inégalités
de la forme  $\abs{\phi}<1$, 
où $\phi\in\mathscr O_{Y,y}$ vérifie $\abs{\phi(y)}=1$,
et $\tilde\Theta$ est la conjonction des inégalités $\abs{\tilde{\phi(y)}}<1$
(pour les mêmes fonctions~$\phi$)
où $\tilde {\phi(y)}\in \hrt y$ est considéré comme un élément
de $\mathscr O_{\tilde {\mathscr Y},\zeta}$.
Rappelons que la valuation~$v$ est la composée de la valuation associée à~$y$
et de~$\tilde v$; comme $\abs{\phi(y)}=1$, les conditions
$\abs{\phi(v)}<1$ et $\abs{\tilde\phi(\tilde v)}<1$ sont équivalentes.
Le lemme en résulte
\end{proof}

Terminons ce paragraphe par un complément sur la structure
des cônes tropicaux lorsqu'ils ne sont pas de dimension maximale
qui indique néanmoins qu'elles ne peuvent être arbitraire.
En particulier, il résultera du théorème~\ref{theo.interieur-harmonieux}
qu'ils peuvent être munies d'un calibrage partout strictement
positif et harmonieux.

\section{Variation des tropicalisations}
 
\begin{lemm}\label{lemm.polconstant2}
Soit $X$ un espace analytique compact,
soit $f\colon X\to Y$ un morphisme d'espaces analytiques,
soit $g\colon X\to \gm^n$ un moment.
Pour tout $y\in Y$, on pose $X_y=f^{-1}(y)$.
Soit $y$ un point \emph{rigide} de~$Y$.
Il existe un voisinage~$U$ de~$y$ dans~$Y$
et un polyèdre~$\Delta\subset g_\trop(X_y)$ 
de dimension~$\leq \dim(X_y)-1$
tels que
$g_\trop(X_u)\subset g_\trop(X_y)$ 
et $g_\trop(X_u\cap \partial(X))\subset\Delta$
pour tout $u\in U$.
\end{lemm}
\begin{proof}
Le domaine analytique~$g_\trop^{-1}(g_\trop(X_y))$ de~$X$ 
contient le fermé de Zariski~$X_y$, donc en est un voisinage
(lemme~\ref{lemm.voisinage}).
Comme $X$ est compact, il existe un voisinage~$U$ de~$y$
tel que $f^{-1}(U)\subset g_\trop^{-1}(g_\trop(X_y))$.
Ainsi, si $u\in U$, 
$g_\trop(X_u)$ est contenu dans~$g_\trop(X_y)$.
 
Le bord~$\partial(X)$ de~$X$ 
est contenu dans une réunion finie de compacts qui sont
chacun de la forme~$\phi^{-1}(\eta_r)$,
où $\phi$ est un morphisme d'un domaine affinoïde~$V$
de~$X$ vers~$\A^1$ et $r$ est un nombre réel~$>0$
(\cite{ducros2012b}, lemme 3.1).
Soit~$\phi \colon V\to \A^1$ un tel morphisme et soit~$r>0$.
Posons $Q=g_\trop(X_y\cap \phi^{-1}(\eta_r))$
et $W=g_\trop^{-1}(Q)\cap \phi^{-1}(\eta_r)$. L'espace~$W$
est un domaine analytique compact de~$V\cap \phi^{-1}(\eta_r)$
et en contient le fermé de Zariski~$X_y\cap \phi^{-1}(\eta_r)$.
Par suite, $W$ est un voisinage de $X_y\cap\phi^{-1}(\eta_r)$
dans~$X\cap\phi^{-1}(\eta_r)$.
Puisque $X$ est compact, et quitte à diminuer~$U$,
on a $g_\trop(X_u\cap \phi^{-1}(\eta_r))\subset Q$ pour $u\in U$.

Par ailleurs, démontrons que $\dim(Q)\leq \dim(X_y)-1$.
Soit~$x\in X_y \cap \phi^{-1}(\eta_r)$;
on a $d_k(x)\leq \dim(X_y)$;
d'autre part, $d_k(x)=d_{\hr{\eta_r}}(x)+1$ ; par conséquent, 
$d_{\hr{\eta_r}}(x)\leq \dim(X_y)-1$.  Ainsi,
la dimension~$\hr {\eta_r}$-analytique
de~$X_y\cap \phi^{-1}(\eta_r)$ est 
majorée par~$\dim(X_y)-1$ et, \emph{a fortiori},
$\dim(Q)\leq \dim(X_y)-1$.

Le lemme en résulte.
\end{proof} 
  
\begin{theo}\label{theo.tropcont}
Soit $d$ un entier,
soit $X$ un espace analytique compact purement de dimension~$n$ qui
satisfait à la condition~$(\mathrm S_d)$,
soit $Y$ un espace analytique régulier purement de dimension~$d$ 
et soit $f\colon X \to Y$  un morphisme purement de dimension relative~$n-d$;
pour tout $t\in Y$, on pose $X_t=f^{-1}(t)$.

Soit $g\colon X\to \gm^N$ un moment.
Soit $y$ un point \emph{rigide} de~$Y$ et
soit $P$ le polyèdre calibré $g_\trop(X_y)^{(n-d)}$.

\begin{enumerate}
\item Il existe un voisinage ouvert~$U$ de~$y$ dans~$Y$ 
tel que, pour tout point~$u\in U$,
le polyèdre calibré $g_\trop(X_u)^{(n-d)}$ soit égal à~$P$.

\item 
On suppose de plus que $Y=\A^d$ et que $y$ est l'origine.
Posons $X'=f^{-1}(\gm^d)$ et soit $h\colon X'\to \gm^N\times\gm^d$
le moment défini par $(g,f)$; c'est un moment agréable.
Il existe un nombre réel~$r>0$ tel que
le polyèdre calibré $h_\trop(X')\cap \R^N\times \mathopen]-\infty,\log(r)\mathclose[^d $
soit égal au produit des polyèdre calibrés~$P$
et de la cellule $\mathopen]-\infty,\log(r)\mathclose[^d$
munie de son calibrage standard.
%
%
%
%
%
\end{enumerate}
\end{theo}
Dans une première
version de ce texte, l'énoncé de ce théorème supposait $d=1$.
Comme l'a noté A.~Mihatsch dans~\cite{Mihatsch-2024},
la démonstration s'étend à tout~$d$.
\begin{proof}
D'après le lemme~\ref{lemm.polconstant2},
il existe un voisinage ouvert~$U$ de~$y$ dans~$Y$ 
et un polyèdre~$\Delta$ de dimension~$\leq n-d-1$
tels que $g_\trop(X_u)\subset g_\trop(X_y)$ 
et $g_\trop(X_u\cap \partial(X))\subset \Delta$ pour tout $u\in U$.
Quitte à le remplacer par la composante connexe de~$y$ dans~$U$,
on suppose que $U$ est connexe.

Choisissons une décomposition
 cellulaire~$\mathscr C$ de~$g_\trop(X_y)$ telle que 
$\Delta$ soit contenu dans la réunion des $(n-d-1)$-cellules de~$\mathscr C$.

%



\begin{lemm}
Soit $C$ une $(n-d)$-cellule de~$\mathscr C$.
Soit $p\colon\gm^N\ra\gm^{n-d}$ un morphisme affine de tores
tel que la restriction à~$C$ de~$p_\trop$ soit injective;
soit $\sigma_C\colon \R^{n-d}\ra\R^N$
l'unique section affine de~$p$ dont l'image contient~$C$.
Posons $W=g_\trop^{-1}(C)$; c'est un domaine de Weierstrass de~$X$,
purement de dimension~$n$.

\begin{enumerate}
\item
Il existe un entier naturel~$\delta_C\neq0$ 
tel que le morphisme
$(p\circ g,f)|_W\colon W\to \gm^{n-d}\times \gm^d$
soit fini et plat de degré~$\delta_C$
au-dessus de tout point de~$\sq_{U}(p_\trop(\mathring C))$.

\item
Pour tout $u\in U$, on a $C\subset g_\trop(X_u)$.
\end{enumerate}
\end{lemm}
(Rappelons,
voir~\S\ref{sq_X(P)},
que l'ensemble $\sq_{U}(p_\trop(\mathring C))$ est constitué des points
de~$X$ de la forme $\sq_u(c)$, pour $u\in U$ 
et $c\in p_\trop(\mathring D)$.)

\begin{proof}
\begin{enumerate}
\item
Soit~$u\in U$ et soit~$S_u$ le squelette de~$(\gm^{n-d})_u$,
identifié à une partie de $\gm^{n-d}\times Y$.
Soit~$s$ un point de l'image canonique~$\sq_u(p_\trop(\mathring C))$ 
de~$p_\trop(\mathring C)$ dans~$S_u$.  
Prouvons que $(p\circ g,f)|_Y$ est fini et plat au-dessus de~$s$.

Soit $x$ un point de~$W$ tel que $(p\circ g, f)(x)=s$.
Le lemme~\ref{lemm.fibre-dim-zero}, appliqué au morphisme
de~$W_u$ dans  $\gm[u]^{n-d}$ induit par~$p\circ g$,
affirme que la dimension de ce morphisme en~$x$ est nulle.
Cela entraîne que  $\dim_x(p\circ g,f)=0$.
%
%
Puisque $\partial(W)\subset \partial(X)\cup g_\trop^{-1}(\partial C)$,
on a $x\not\in\partial(W)$.
Puisque $f_W$ est propre, 
le lemme~\ref{lemme.finiplat-global}, \emph a), entraîne que
$f|_W$ est fini au-dessus d'un voisinage de~$s$.

%

Comme~$d_{\hr t}(s)=n-d$,  on a $d_k(s)\geq n-d$. 
Les espaces~$W$ et~$\gm^{n-d}\times Y$ sont
tous deux purement de dimension~$n$, 
l'espace $k$-analytique~$\gm^{n-d}\times Y$ est régulier.
Comme $X$ vérifie la propriété~($\mathrm S_d$),
le lemme~\ref{plat.si} appliqué au morphisme~$(p\circ g,f)|_W$ entraîne que 
ce morphisme est plat au-dessus de~$s$. 

Comme $U$ est connexe,
l'ensemble $\sq_{U}(p_\trop(\mathring C))=\{\sq_u(p_\trop(\mathring C))\}_{u\in U}$ est connexe, d'après~\ref{sq_X(P)}. 
Par suite, le degré de~$(p\circ g, f)|_W$ au-dessus de~$s$ 
ne dépend pas du point~$s$ choisi.
Notons-le~$\delta_C$.

 


Par hypothèse, $C\subset g_\trop(X_y)$, 
donc $p_\trop(C)\subset p\circ g_\trop(W_y)$;
d'après le théorème~\ref{theo.proprietes-Sigmaf}, 
on a $p_\trop(C)\subset p\circ g_\trop(\Sigma_{(p\circ g)_y}^{(n-d)})$ 
donc tout point de~$\sq_y(p_\trop(\mathring C))$ 
possède un antécédent par~$p\circ g$.
Cela entraîne que $\delta_C\neq 0$ et termine la démonstration
de l'assertion~\emph a).

\item
Soit $s$ un point de~$ \sq_{U}(p_\trop(\mathring C))$ 
et soit~$u$ son image sur~$U$. 
Par changement de base, 
le morphisme de~$W_u$ dans~$(\gm^{n-d})_u$ déduit de~$p\circ g$
est fini et plat de degré~$\delta_C$ 
au-dessus de~$s$.

Comme $\delta_C\neq 0$, 
il s'ensuit que~$p_\trop(\mathring C)\subset p\circ g_\trop(W_u)$ 
pour tout~$u\in U$. 
Par densité de~$\mathring C$ dans~$C$
et compacité de $W_u$, 
on a $p_\trop(C)\subset p \circ g_\trop(W_u)$ pour tout~$u\in U$.
Par définition de~$W$, ceci entraîne que  $C\subset g_\trop(W_u)$ 
pour tout $u\in U$.
\qedhere
\end{enumerate}
\end{proof}

Reprenons la démonstration du théorème~\ref{theo.tropcont}.

\emph a)
Démontrons d'abord la constance des \emph{polyèdres calibrés} 
$g_\trop(X_u)^{(n-d)}$ pour $u\in U$.

Soit $u\in U$.

D'après le lemme précédent,
toute $(n-d)$-cellule de~$g_\trop(X_y)$ est
contenue dans~$g_\trop(X_u)$. Par suite,
$g_\trop(X_y)^{(n-d)}\subset g_\trop(X_u)$.
Par définition de~$U$, on a également $g_\trop(X_u)\subset g_\trop(X_y)$.
Il en résulte que $g_\trop(X,u)^{(n-d)}=g_\trop(X_y)^{(n-d)}$.

De plus,  le calibrage
canonique de $g_\trop(X_u)$ est adapté 
à la décomposition cellulaire~$\mathscr C$,
car
aucune $(n-d)$-cellule ouverte de~$\mathscr C$
ne rencontre $g_\trop(X_u\cap \partial (X))$.

Soit $C$ une $(n-d)$ cellule de~$\mathscr C$,
soit $p\colon\gm^N\to \gm^{n-d}$ un morphisme affine de tores
tel que la restriction à~$C$  de~$p_\trop$ soit
injective et soit $\sigma_C\colon\R^{n-d}\to\R^N$
l'unique section affine de~$p$ dont l'image contient~$C$.
Par définition des calibrages canoniques,
le vecteur-volume $\mu_{C,u}$ de $g_\trop(X_u)$ sur~$C$ est donné par
\[ \mu_{C,u} = \delta_C \sigma_{C,*}(\abs{e_1\wedge\dots\wedge e_{n-d}}). \]
Ainsi, le polyèdre calibré $g_\trop(X_u)^{(n-d)}$
est égal à $g_\trop(X_y)^{(n-d)}=P$, ainsi qu'il fallait démontrer.

\emph b)
Nous supposons maintenant que $Y=\A^d$ et que $y=0$ est l'origine.

Le morphisme $f\colon X\to Y$  est topologiquement propre, 
donc le morphisme $f|_{X'}\colon X'\to\gm^d$ est topologiquement propre.
Par suite, le moment $h'=(g|_{X'},f|_{X'})$ est topologiquement propre,
et en particulier agréable.

On choisit un nombre réel~$r>0$ assez petit
tel que l'on puisse prendre $U=D(0,r)$  dans~\emph a).

Posons~$\Gamma=\mathopen]-\infty;\log(r)\mathclose[^d$
et notons $X_\Gamma = X'\cap f^{-1}(D(0,r))=X'\cap f_\trop^{-1}(\Gamma)$.

Prouvons d'abord que 
$h'_\trop(X_\Gamma) = P \times \Gamma$.

Soit $\gamma\in \Gamma$
et soit $u\in D(0,r)$ tel que $\trop(u)=\gamma$.
Soit $C$ une cellule de dimension~$(n-d)$ de~$\mathscr C$.
La cellule $C\times\{\gamma\}$ est contenue dans $g_\trop(X_u)$,
d'après~\emph a). Par suite, 
la cellule $C\times \Gamma$
est contenue dans~$h'_\trop(X_\Gamma)$.
Comme elle est de dimension~$n$,
elle est contenue dans~$h'_\trop(X_\Gamma)^{(n)}$.
Inversement, la définition 
de~$U$ au début de la démonstration garantit
que pour tout $u\in D(0,r)\cap \gm^d$,
$g_\trop(X'_u) \subset g_\trop(X_y)$.  Par suite, 
$h'_\trop(X_\Gamma)\subset g_\trop (X_y) \times \Gamma$.
On a donc 
$h'_\trop(X_\Gamma)^{(n)}\subset g_\trop (X_y)^{(n-d)} \times \Gamma$.
On a donc l'égalité voulue.

Démontrons maintenant l'égalité des calibrages.
Soit $C$ une $(n-d)$-cellule de~$\mathscr C$.
notons aussi 
$q\colon \gm^N\times\gm^d\ra\gm^{n}$ le morphisme
de tores $(p,\id)$. 
L'application affine $\tau_C=(\sigma_C,\id)$
est l'unique section affine de~$q_\trop$ dont l'image contient $C\times \Gamma$.
Nous allons démontrer que le calibrage de la cellule~$C\times\Gamma$ 
de $h'_\trop(X_\Gamma)^{(n)}$ est au produit
du calibrage~$\mu_C$ de~$C$ dans~$P$ par le calibrage standard de~$\Gamma$.

Posons $X_{C,\Gamma}=(h'_\trop)^{-1}(C\times\Gamma)$;
nous allons démontrer que le morphisme $q\circ h'|_{X_{C,\Gamma}}$ est fini
et plat de degré~$\delta_C$ au-dessus de tout point
de~$\sq(q_\trop(\mathring C\times \Gamma))$.

Soit $c\in p_\trop(\mathring C)$ et $\gamma\in\Gamma$ 
Par la construction de la décomposition cellulaire~$C$ et le choix de~$r$,
le point~$\sq(c,\gamma)$ n'appartient pas à $q\circ h'(\partial(X_{C,\Gamma}))$.
D'après le lemme~\ref{lemme.finiplat-global}, \emph c),  
le morphisme $q\circ h'|_{X_{C,\Gamma}}$ est fini et plat au-dessus de~$\sq(c,\gamma)$; notons $\delta$ son degré.

On a $q\circ h'=(p,\id) \circ (g,f)|_{X'}=(p\circ g, f)|_{X'}$.
Par changement de base, le morphisme
de $X_{C,\Gamma,\sq(\gamma)}$ dans~$\gm[\mathscr H(\sq(\gamma))]^{n-d}$
déduit de~$p\circ g$
est fini et plat de degré~$\delta$ au-dessus de~$\sq_{\mathscr H\sq(\gamma)}(
\sq(\gamma))$. 
D'après la première partie de l'énoncé et le choix de~$r$, 
ce degré est égal à~$\delta_C$.

Par conséquent (\cf \S\ref{ss.cal-can-but}),
le calibrage de la cellule $C\times\Gamma$
est constant, de vecteur-volume
\[ \mu_{C\times \Gamma} = \delta_C \tau_{C,*} (\abs{e_1\wedge\dots\wedge e_{n-1}\wedge e_n}). \]
De plus,
\[ \tau_{C,*} (\abs{e_1\wedge\dots\wedge e_{n-1}\wedge e_n)}
= \sigma_{C,*} (\abs{e_1\wedge\dots\wedge e_{n-d}}) 
\cdot \abs{ e_{n-d+1}\wedge\dots \wedge e_n}, \]
qui est le produit du vecteur-volume associé à~$C$ dans~$P$
et du vecteur-volume associé au calibrage standard de~$\Gamma$.
Ceci achève la démonstration du théorème.
\end{proof} 

\begin{rema}
Le théorème~\ref{theo.tropcont} ne peut pas être vrai
sans certaines hypothèses sur le morphisme~$f\colon X\to Y$.

Considérons par exemple le cas où $f$ est la normalisation
d'une courbe lisse en dehors d'un unique point double ordinaire~$y$
et $g$ est le moment trivial ($N=0$).
Dans ce cas, la tropicalisation de~$X_y$ est un point avec multiplicité~$2$.
En revanche, pour tout autre point $u\neq y$, la tropicalisation
de~$X_u$ est un point avec multiplicité~$1$.
Dans ce cas, la base~$Y$ n'est pas régulière.

De même, considérons dans l'éclaté du bidisque en l'origine
le domaine analytique défini par $\abs{y/x}\leq 1$.
Son image par l'application d'éclatement
est définie par l'inégalité $\abs{y} \leq \abs x$.
Sur tout voisinage de l'origine, certaines fibres sont vides et d'autre non.
La fibre de~$f$ en l'origine est de dimension~$1$ et les autres 
sont de dimension~$\leq0$.
\end{rema}

\def\pf{\text{\upshape pf}}
\def\degcal{\operatorname{dc}\nolimits}
\chapter{Formes différentielles réelles en géométrie ultramétrique}
\label{chap.fd-berk}

Soit $X$ un espace $k$-analytique  topologiquement séparé.
La définition~\ref{defi.calibrage-canonique} le munit
d'une structure d'espace G-tropical affiné.

On en déduit des notions de formes et G-formes différentielles
de type~$(p,q)$ sur~$X$, à coefficients dans une classe de coefficients.

Au \S\ref{ss.an-G-tropical-vertebre}, nous avons expliqué
comment ses parties squelettiques le munissent d'une structure
d'espace G-tropical vertébré de niveau~$n$.
La définition~\ref{defi.calibrage-canonique} munit
les parties squelettiques de~$X$ d'un $n$-calibrage,
de sorte que $X$ est muni d'une structure 
d'espace G-tropical calibré de niveau~$n$.

On en déduit
une théorie de l'intégration des $(n,n)$-formes et des $(n-1,n)$-formes
sur~$X$, ainsi qu'une théorie de courants.

Dans ce chapitre, nous explorons plus avant les propriétés
des formes différentielles sur les espaces analytiques.

Nous allons également démontrer que les espaces analytiques sans bord
sont harmonieux et munir les espaces analytiques purement de dimension~$n$
d'une structure d'espace tropical.

\section{Formes et fonctions lisses}

\begin{prop}
Soit $X$ un espace $k$-analytique topologiquement séparé,
considéré comme espace G-tropical affiné.

\begin{enumerate}
\item
Soit $\omega$ une G-forme de type~$(p,q)$ sur~$X$.
Il existe un G-recouvrement $(V_i)$ de~$X$ et, pour tout~$i$,
un \emph{moment} $f_i\colon V_i\to T_i$ sur~$V_i$,
une forme~$\alpha_i$ de type~$(p,q)$ sur une partie paralinéaire
de~$T_{i,\trop}$ contenant $f_{i,\trop}(V_i)$
tels que $\omega|_{V_i}=f_{i,\trop}^*\alpha_i$.

\item
Soit $\omega$ une forme de type~$(p,q)$ sur~$X$.
Il existe un recouvrement \emph{ouvert} $(V_i)$ de~$X$ et, pour tout~$i$,
un \emph{moment} $f_i\colon V_i\to T_i$ sur~$V_i$,
une forme~$\alpha_i$ de type~$(p,q)$ sur une partie paralinéaire
de~$T_{i,\trop}$ contenant $f_{i,\trop}(V_i)$
tels que $\omega|_{V_i}=f_{i,\trop}^*\alpha_i$.
\end{enumerate}
\end{prop}
\begin{proof}
On  ne traite que le cas des G-formes, l'autre est analogue.
Par hypothèse, il existe un G-recouvrement~$(W_i)$ de~$X$
et, pour tout~$i$, une application paralinéaire~$g_i$ sur~$W_i$
à valeurs dans un espace affine~$\R^{n_i}$, 
une forme~$\beta_i$ de type~$(p,q)$ sur une partie paralinéaire
de~$E_i$ contenant~$g_i(W_i)$, tels que
$\omega|_{W_i}=g_i^*\beta_i$ pour tout~$i$.
Par définition d'une application paralinéaire, il existe
pour tout~$i$ un G-recouvrement~$(V_{i,j})$ de~$W_i$
et, pour tout~$j$, un moment $f_{i,j}$ sur~$V_{i,j}$,
à valeurs dans le tore~$\gm^{n_{i}}$ 
un nombre rationnel $a_{i,j}$ et un élément~$c_{i,j}\in\R^{n_i}$
tels que $g_{i,j}=a_{i,j} f_{i,j,\trop}+c_{i,j}$.
Notons $p_{i,j}\colon \R^{n_i}_\trop \to \R^{n_i}$
l'application affine $t\mapsto a t + c_{i,j}$;
on a donc $g_{i,j}=p_{i,j}\circ f_{i,j,\trop}$, de sorte que
\[ \omega|_{V_{i,j}}= f_{i,j,\trop}^* (p_{i,j}^* \beta_i). \]
L'assertion en résulte.
\end{proof}

\begin{coro}
Une fonction $u$ sur~$X$ est lisse au voisinage 
d'un point~$x$ s'il existe un voisinage~$U$ de~$x$ dans~$X$,
des fonctions inversibles $f_1,\dots,f_n$ sur~$U$
et une fonction~$\phi$ sur~$\R^n$ au voisinage
de $(f_{1,\trop}(x),\dots,f_{n,\trop}(x))$
telle que $u=\phi\circ f_\trop$ sur~$U$.
\end{coro}

Dans la suite du texte, nous travaillerons essentiellement
avec ce type de cartes (G)-tropicales.

\begin{prop}
Soit $\mathscr G$ une classe de coefficients.
Soit $X$ un espace analytique topologiquement séparé
et soit~$u$ une fonction numérique sur~$X$. 

\begin{enumerate}
\item Pour que $u$ soit G-localement de classe~$\mathscr G$,
il faut et il suffit que $u|_{X_\red}$ le soit.

\item Si $X$ est presque bon, pour que $u$ soit localement
de classe~$\mathscr G$, il faut et il suffit que $u|_{X_\red}$ le soit.
\end{enumerate}
\end{prop}
Une variante de cet énoncé vaut pour les formes et les G-formes.
\begin{proof}
Si $u$ est localement (resp. G-localement) de classe~$\mathscr G$,
il en est de même de sa restriction à~$X_\red$. Démontrons la réciproque.

\begin{enumerate}
\item
Supposons que $u|_{X_\red}$ soit G-localement de classe~$\mathscr G$.
Soit $V$ un domaine affinoïde de~$X_\red$ sur lequel il existe
un moment $f\colon V\to \gm^n$ et une fonction~$\phi$ 
de classe~$\mathscr G$ sur~$\R^n$
telle que $u|_V= \phi \circ f_\trop$. 
On déduit du théorème de Gerritzen--Grauert 
(voir par exemple \cite{temkin2005}, théorème~3.1) que
$V$ est recouvert par une réunion de domaines affinoïdes
de la forme~$W_\red$, où $W$ est un domaine affinoïde de~$X$.
Sur $W_\red$, chacune des~$f_i$ se relève en une fonction $g_i$ sur~$W$;
ces fonctions sont inversibles et ont même norme que~$f_i$;
elles définissent un moment $g\colon W\to \gm^n$ tel que 
$u_W=\phi\circ g_\trop$. Cela prouve que $u$ est G-localement de classe~$\mathscr G$.

\item
Soit $x$ un point de~$X$ et soit $V$ un voisinage analytique de~$x$
dans~$X_\red$ sur lequel il existe 
un moment $f\colon V\to \gm^n$ et une fonction~$\phi$ de classe~$\mathscr G$
sur~$\R^n$
telle que $u|_V= \phi \circ f_\trop$. 
Comme $X$ est presque bon en~$x$, chaque~$f_i(x)\in\hr x$
est approchée par une fonction $g_i\in \mathscr O_{X,x}$.
En particulier $f_i$ et $g_i$ ont même norme au voisinage de~$g$.
Ainsi, les $g_i$ définissent un moment~$g$ défini au voisinage
de~$X$ sur lequel on a $u=\phi\circ g_\trop$; cela prouve
que $u$ est de classe~$\mathscr G$ au voisinage de~$x$.
\qedhere
\end{enumerate}
\end{proof}

\section{Séparation holomorphe}

\begin{prop}\label{lemm.support-comp-ferme-bk}
Soit $X$ un espace $k$-analytique topologiquement séparé,
considéré comme espace G-tropical affiné.
Soit $\omega$ une G-forme de type~$(p,q)$ sur~$X$.
\begin{enumerate}
\item
Le support de~$\omega$ est contenu dans l'ensemble des points~$x\in X$
tels que $d_k(x)\geq \max(p,q)$.
\item
Soit $V$ un domaine analytique de~$X$ et soit $Y$ un fermé de Zariski
de~$V$. Si $\dim(Y)<\max(p,q)$, le support de~$\omega$ ne rencontre pas~$Y$.
\end{enumerate}
\end{prop}
\begin{proof}
Démontrons l'assertion~\emph a).
Soit $x\in X$ un point tel que $d_k(x)<\max(p,q)$.
Comme $d_\trop(x)\leq d_k(x)$, on a donc $d_\trop (x)< \max(p,q)$;
l'assertion résulte donc de la proposition~\ref{lemm.support-comp-ferme}.

L'assertion~\emph b) en résulte car tout point~$y$ de~$Y$  
vérifie $d_k(y)\leq\dim(Y)$.
\end{proof}

\subsection{Séparation holomorphe}\label{hol-sep}
On dira qu'un espace analytique~$X$ est
\emph{holomorphiquement séparé} si
pour tout couple~$(x,y)$ de points de~$X$ avec~$x\neq y$
il existe une fonction analytique~$f$ sur~$X$ 
telle que~$\abs{f(x)}\neq \abs{f(y)}$.  Un tel espace
est topologiquement séparé.

Si~$X$ est un espace holomorphiquement séparé, il en va 
trivialement de même de tout domaine 
analytique de~$X$. 

Tout espace affinoïde est par définition 
holomorphiquement séparé ; en conséquence,
tout domaine analytique d'un espace affinoïde est holomorphiquement séparé. 
\index{espace ! holomorphiquement séparé}

Un espace analytique sera dit \emph{localement holomorphiquement
séparé} s'il est topologiquement séparé et 
si chacun de ses points possède un voisinage analytique 
holomorphiquement séparé. 

Un espace analytique topologiquement séparé
qui est bon, ou plus généralement dont 
tout point possède un voisinage isomorphe à un domaine
analytique d'un espace affinoïde, est localement holomorphiquement
séparé.  C'est en particulier le cas d'un espace analytique 
topologiquement séparé sans bord.

\index{espace analytique!localement holomorphiquement séparé}

\begin{lemm}
Soit $X$ un espace $k$-analytique holomorphiquement séparé. 
L'anneau des fonctions lisses sur~$X$ sépare les points.
\end{lemm}
\begin{proof}
Soit $x$, $y$ des points distincts de~$X$.
Par définition de la séparation
holomorphe, 
il existe une fonction
analytique~$f$ sur~$X$ telle
que $\abs{f(x)}\neq\abs{f(y)}$.

Supposons ces deux valeurs non nulles ; soit $\phi$
une fonction $\mathscr C^\infty$ à support compact sur~$\R$
valant~$0$ au voisinage de~$\log\abs{f(x)}$ et~$1$ au
voisinage de~$\log\abs{f(y)}$. 
La fonction $\phi\circ\log\abs f$, définie sur l'ouvert 
d'inversibilité de~$f$,
est lisse; prolongeons-la par~$0$ sur~$X$ tout entier. Comme $\phi$ 
est
à support compact, la fonction obtenue est nulle au voisinage
de l'ensemble des zéros de~$f$; elle est donc lisse
et prend des valeurs différentes en~$x$ et~$y$.

Si $f(x)=0$, on a alors $f(y)\neq 0$
et on raisonne de même à partir d'une fonction~$\phi$
à support compact valant~$1$ au voisinage de~$\log\abs{f(y)}$.

Le cas où $f(y)=0$ est analogue par symétrie.
\end{proof}

\begin{coro}\label{coro.berk-local-diff-sep}
Soit $X$ un espace $k$-analytique localement holomorphiquement séparé. 
L'espace G-tropical affiné~$X$ est localement différentiellement séparé.
En particulier,
le faisceau des fonctions lisses sur~$X$ sépare localement les points.
\end{coro}

\begin{coro}\label{coro.berk-lisse-separe}
Soit $X$ un espace analytique localement holomorphiquement séparé.
Soit $K$ une partie compacte de~$X$ et soit $U$ un voisinage ouvert de~$K$.
Il existe une fonction lisse sur~$X$ qui vaut~$1$ au voisinage de~$K$
et dont le support est compact et contenu dans~$U$.
\end{coro}
\begin{proof}
Comme $X$ est localement différentiellement séparé,
c'est une conséquence du corollaire~\ref{coro.trop-lisse-separe}.
\end{proof}

\begin{theo}[Stone--Weierstraß] \label{theo.berk-sw}
Soit $X$ un espace analytique localement holomorphiquement séparé
et soit $U$ un ouvert de~$X$.
Soit $f$ une fonction continue à valeurs réelles sur~$X$
dont le support est compact et contenu dans~$U$.
Soit $a$ et $b$ des nombres réels tels que $a\leq f\leq b$.
Pour tout~$\eps>0$, il existe une fonction lisse~$g$ sur~$X$
dont le support est compact, contenu dans~$U$,
telle que $\abs{f(x)-g(x)}<\eps$ pour tout $x\in X$
et telle que $a\leq g\leq b$.
\end{theo}
\begin{proof}
Comme $X$ est localement différentiellement séparé,
c'est une conséquence du théorème~\ref{theo.trop-sw}.
\end{proof}

\begin{coro}\label{coro.berk-sw}
Soit $X$ un espace analytique localement holomorphiquement séparé.

\begin{enumerate}
\item L'ensemble des fonctions lisses à support compact sur~$X$
est dense dans l'espace~$\mathscr C_\cpct(X)$
des fonctions continues à support compact sur~$X$, 
muni de la topologie limite inductive des espaces
de fonctions à support dans un compact donné.
\item L'ensemble des fonctions lisses à support compact sur~$X$
est dense dans l'espace~$\mathscr C(X)$ des fonctions continues sur~$X$
muni de la topologie de la convergence uniforme sur tout compact.
\end{enumerate}
\end{coro}
\begin{proof}
Comme $X$ est localement différentiellement séparé,
c'est une conséquence du corollaire~\ref{coro.trop-sw}.
\end{proof}

\begin{prop}[Partitions de l'unité lisses]
\label{prop.berk-partition.1.lisse}
Soit $X$ un espace analytique localement holomorphiquement séparé,
paracompact.
Soit $(U_i)_{i\in I}$ un recouvrement ouvert de~$X$.
Il existe une partition de l'unité\index{partition de l'unité} 
subordonnée à~$(U_i)$
formée de fonctions lisses sur~$X$.
\end{prop}
\begin{proof}
Comme $X$ est localement différentiellement séparé,
cela résulte de la proposition~\ref{prop.trop-partition.1.lisse}.
\end{proof}

\begin{coro}
\label{prop.berk-partition.1.lisse-compact}
Soit $X$ un espace analytique localement holomorphiquement séparé et
soit $F$ une partie fermée de~$X$ qui est dénombrable à l'infini
(par exemple, compacte).
Soit $(U_i)_{i\in I}$ un recouvrement de~$K$ par des ouverts de~$X$.
Il existe une famille localement finie $(\theta_i)$ 
de fonctions lisses sur~$X$, à support compact, positives,
telle que pour tout~$i$, le support de~$\theta_i$ soit contenu dans~$U_i$
et que $\sum\theta_i$ vaille~$1$ au voisinage de~$F$.
\end{coro}
\begin{proof}
Comme $X$ est localement différentiellement séparé,
cela résulte du corollaire~\ref{prop.trop-partition.1.lisse-compact}.
\end{proof}

\begin{coro}\label{coro.berk-lisse-fin}
Soit $X$ un espace analytique localement holomorphiquement séparé, paracompact.
Les faisceaux $\mathscr A^{p,q}_X$ de  formes lisses sur~$X$ sont fins.
\end{coro}
\begin{proof}
Comme $X$ est localement différentiellement séparé,
cela résulte du corollaire~\ref{coro.pl-fin}.
\end{proof}

\section{Extension des scalaires}

Soit $X$ un espace $k$-analytique de niveau~$n$.
On le considère comme un espace G-tropical de niveau~$n$.
On a montré au \S\ref{defi.integrale.n,n}
comment toute $(n,n)$-forme~$\omega$ sur~$X_\groth$
à coefficients tropicalement boréliens définit
une mesure de Borel~$\omega_\Bor$ sur~$X$,
et lorsque cette G-forme est localement intégrable,
une mesure de Radon~$\omega_\Rad$.

\begin{prop}\label{prop.mesure-chgt-base}
Soit $L$ une extension complète de~$k$, soit $p\colon X_L\to X$
le morphisme de changement de base.
Pour toute forme~$\omega$ de type~$(n,n)$ sur~$X_\groth$
à coefficients tropicalement boréliens,
on a $p_*(\omega_{L,\Bor})=\omega_{\Bor}$.
Lorsque $\omega$ est localement intégrable,
on a $p_*(\omega_{L,\Rad})=\omega_{\Rad}$.
\end{prop}
\begin{proof}
En revenant aux définitions, on déduit de la compatibilité
du calibrage canonique à l'extension des scalaires
(proposition~\ref{prop.calibrage-chgt-base})
l'égalité $p_*\widetilde{\omega_L}=\omega_L$ de formes-volume
$n$-dimensionnelles. La proposition en résulte.
\end{proof}

\begin{prop}\label{prop.mesure-chgt-base-bord}
Soit $L$ une extension complète de~$k$, soit $p\colon X_L\to X$
le morphisme de changement de base.
Pour toute forme~$\omega$ de type~$(n-1,n)$ sur~$X_\groth$
à coefficients tropicalement boréliens,
on a $p_*(\omega_{L,\Bor})=\omega_{\Bor}$.
Lorsque $\omega$ est localement intégrable,
on a $p_*(\omega_{L,\Rad})=\omega_{\Rad}$.
\end{prop}
\begin{proof}
En revenant aux définitions, on déduit de la compatibilité
du calibrage canonique à l'extension des scalaires
(proposition~\ref{prop.calibrage-chgt-base})
l'égalité $p_*\widetilde{\omega_L}=\omega_L$ de formes-volume
$(n-1)$-dimensionnelles. La proposition en résulte.
\end{proof}

\section{Intégrales et morphismes finis}

\subsection{}
Soit $n\in \N$ et soit $X$ et~$Y$ des espaces $k$-analytiques de niveau~$n$
et soit $f\colon Y\to X$ un morphisme qui est fini
au-dessus du complémentaire d'un fermé de Zariski de~$X$ de dimension~$<n$.
On a défini au \S\ref{ss.degcal2} une fonction $x \mapsto \degcal^x(f)$
de~$X$ dans~$\N$; nous la noterons~$\degcal_f$.

Soit $p\in\{n-1,n\}$ et soit $\omega$ une $G$-forme de type $(p,n)$  
sur~$X$. Posons $\eta=f^*\omega$.

\begin{lemm}\label{lemm.mesure-f*-degre-integral}
Soit $\Sigma$ un partie squelettique de~$X$ qui supporte
fortement~$\omega$;
notons $\tilde\omega$ la forme-volume $p$-dimensionnelle sur~$\Sigma$
déduite de~$\omega$ et du calibrage canonique de~$\Sigma$.

Alors $p^{-1}(\Sigma)$ est une partie squelettique
de~$Y$ qui supporte fortement~$f^*\omega$;
notons $\widetilde{f^*\omega}$
la forme-volume $p$-dimensionnelle sur~$p^{-1}(\Sigma)$
déduite de~$f^*\omega$ et du calibrage canonique de~$p^{-1}(\Sigma)$.
On a l'égalité:
\[ f_* \widetilde{f^*\omega} = \degcal_f \cdot \tilde\omega. \]
\end{lemm}
\begin{proof}
Posons $\Tau=f^{-1}(\Sigma)$. Notons $\mu_\Sigma$ et $\mu_\Tau$
les calibrages canoniques de~$\Sigma$ et~$\Tau$.
On a l'égalité $f_*\mu_\Tau=\degcal_f\cdot \mu_\Sigma$
(proposition~\ref{prop.formule-degre-integral}).
L'assertion découle alors de la proposition~\ref{prop.formes-volume-pl-f*}.
\end{proof}

Compte tenu de la définition des mesures de Borel et, éventuellement
de Radon, associées aux formes volumes $p$-dimensionnelles,
on en déduit la proposition suivante.

\begin{prop}\label{prop.mesure-f*-degre-integral}
Supposons que $\omega$ soit localement intégrable.
Alors, $f^*\omega$ est localement intégrable
et l'on a l'égalité 
\[ f_*(f^*\omega)_\Rad =\degcal_f \cdot \omega_{\Rad} \]
de mesures de Radon sur~$X$. 
\end{prop}
Remarquons que les supports des mesures $\omega_\Rad$
et $f_*(f^*\omega)_\Rad$ évitent tout fermé de Zariski 
d'intérieur vide de~$X$; en particulier, ils sont contenus dans le domaine
de définition du degré calibré.

\section{Existence de tropicalisations (presque) globales}

La définition d'une forme (resp. d'une G-forme)
est locale (resp. G-locale), mais lorsque la source 
est affinoïde, on peut en donner une version partiellement globale ; 
c'est l'objet de cette section. On fixe une classe admissible~$\mathscr F$
de fonctions numériques. Pour la définition d'un système rationnel~$\beta$
d'inéquations, et du système~$\mathring \beta$ associé, nous renvoyons
au paragraphe~\ref{def.domaff}.

\subsection{Description des G-formes et des
formes}
Soit~$X$ un espace~$k$-affinoïde, et soit~$\omega$ une G-forme
sur~$X$ à coefficients dans~$\mathscr F$. Soit~$E$ un ensemble de fonctions
analytiques sur~$X$. Nous dirons que~$\omega$ est {\em G-descriptible au moyen de~$E$}
s'il existe : 
\begin{itemize}
\item un recouvrement affinoïde fini~$(V_i)$ de~$X$ tel que chaque~$V_i$ puisse être décrit
par un système rationnel d'inéquations~$\beta_i$ 
ne mettant en jeu que des fonctions appartenant à~$E$ ; 
\item pour tout~$i$, une famille finie~$h_i$ de fonctions appartenant à~$E$ dont la restriction à~$V_i$ est inversible, et telle que le moment~$h_i|_{V_i}$ tropicalise~$\omega|_{V_i}$.
\end{itemize}

Si~$\omega$ est une forme, nous dirons qu'elle est {\em descriptible au moyen de~$E$} s'il existe~$(V_i,\beta_i,f_i)$ comme ci-dessus
telle que de surcroît
le lieu de validité des inéquations strictes~$\mathring{\beta_i}$ recouvre~$X$.  

\begin{rema}\label{descript-fonction-nonnulle}
Dans la définition ci-dessus, on peut toujours supposer que les~$V_i$ 
sont non vides. On peut alors choisir pour tout~$i$ un système~$\beta_i$ 
ne mettant en jeu que des fonctions non nilpotentes. En effet,
fixons~$i$ et partons d'un système~$\beta_i$ quelconque décrivant~$V_i$, 
système que l'on écrit
\[\abs{f_1}\leq \lambda_1 \abs g\;\text{et}\; \ldots\;\text{et}\; \abs{f_n}\leq \lambda_n \abs g.\]

La fonction~$g$ est inversible sur~$V_i$ qui est non vide, et elle est par conséquent non nilpotente. 
Et si l'une des~$f_j$ est nilpotente, 
l'inégalité~$\abs{f_j}\leq \lambda_j \abs g$ est vraie en tout point de~$X$,
et peut dès lors être retirée du système~$\beta_i$ sans modifier son lieu de validité. 
\end{rema}

\begin{lemm}\label{descript.trop}
Soit~$X$ un espace~$k$-affinoïde et soit~$\omega$ une
G-forme (resp. forme) à coefficients dans~$\mathscr F$ sur~$X$. Soit~$f=(f_1,\ldots, f_n)$
une famille finie de fonctions
\emph{inversibles}
sur~$X$. Si~$\omega$ est G-descriptible
(resp. descriptible) au moyen des~$f_i$, il existe une G-forme
(\resp une forme) $\alpha$ sur $f_\trop(X)$ telle que 
$\omega=f^*\alpha$. 
\end{lemm}

\begin{proof} Posons~$P=f_\trop(X)$.
Commençons par traiter le cas où~$\omega$
est une G-forme. Elle est alors tropicalisée
par un recouvrement affinoïde fini~$(V_j)$ de~$X$, chaque~$V_j$ étant défini
par un système rationnel d'inéquations~$\beta_j$ ne faisant intervenir que les fonctions~$f_i$. 

Par construction, chaque~$V_j$ est l'image réciproque par~$f_\trop$ d'un polytope~$P_j$, 
la forme~$\omega|_{V_j}$ provient d'une forme~$\alpha_j$ sur~$P_j$, et les~$P_j$ recouvrent~$P$. 
La proposition~\ref{prop.controle-tige}
garantit que les~$\alpha_j$ sont uniquement déterminées, et que~$\alpha_j|_{P_j\cap P_\ell}=\alpha_\ell|_{P_j\cap P_\ell}$
pour tout~$(j,\ell)$. En conséquence, les~$\alpha_j$ se recollent en une G-forme~$\alpha$, et~$\omega=f^*\alpha$
par construction. 

Supposons maintenant que~$\omega$ est une forme. On choisit~$V_j,\beta_j, P_j$ et~$\alpha_j$ comme ci-dessus, 
en supposant de plus que les lieux de validité des~$\mathring {\beta_j}$ recouvrent~$X$. Pour tout~$j$, le lieu de validité
de~$\mathring{\beta_j}$ est l'image réciproque d'un ouvert~$Q_j$ de~$P_j$, et les~$Q_j$ recouvrent~$P$. 

On a~$\alpha_j|_{P_j\cap P_\ell}=\alpha_\ell|_{P_j\cap P_\ell}$
pour tout~$(j,\ell)$ ; \emph{a fortiori}, $\alpha_j|_{Q_j\cap Q_\ell}=\alpha_\ell|_{Q_j\cap Q_\ell}$
pour tout~$(j,\ell)$. Les formes~$\alpha_j|_{Q_j}$ se recollent donc en une forme~$\alpha$ sur~$P$, 
et on a par construction~$\omega =f^*\alpha$. 
\end{proof}

\begin{prop}\label{trop.partiel}
Soit $X$ un espace $k$-affinoïde et soit $\omega$ une G-forme 
sur~$X$ à coefficients dans~$\mathscr F$. 
Il existe une famille finie $E$ de fonctions 
holomorphes
sur~$X$ telle que~$\omega$ soit G-descriptible 
au moyen de~$E$, et descriptible au moyen de~$E$ si~$\omega$
est une forme. 
\end{prop}

\begin{proof}
Par les définitions respectives d'une G-forme
et d'une forme, et en vertu des résultats généraux
rappelés au~\ref{def.domaff}, il existe une famille finie~$(\beta_i)$ de
systèmes rationnels d'inéquations
sur~$X$
satisfaisant les conditions suivantes, 
en notant~$V_i$ le
domaine rationnel égal au lieu de validité de~$\beta_i$ : 
\begin{itemize}
\item chacun des~$V_i$ tropicalise~$\omega$ ; 
\item l'espace~$X$ est la réunion des~$V_i$, et même
des lieux de validité des~$\mathring {\beta_i}$ si~$\omega$ est une forme. 
\end{itemize}

Pour tout~$i$, le système~$\beta_i$ est de la forme
\[\abs{f_{i,1}}\leq \lambda_{i,1}\abs{g_i}\;\text{et}\;\ldots\;\text{et}\;\abs{f_{i,n_i}}
\leq \lambda_{i,n_i}\abs{g_i}.\]

La fonction~$g_i|_{V_i}$ est inversible, et les fonctions de la
forme~$(a|_{V_i})/(g_i|_{V_i})$, où~$a$
est une fonction analytique sur~$X$, sont denses dans l'anneau des fonctions
de~$V_i$. 

Par hypothèse,~$\omega|_{V_i}$ est tropicalisée par une famille finie~$(h_{i,\ell})$ de fonctions
inversibles sur~$V_i$. Comme cette dernière propriété ne dépend par définition 
que des fonctions~$\abs{h_{i,\ell}}$,
on peut supposer que chaque~$h_{i,\ell}$ est de
la forme~$(a_{i,\ell}|_{V_i})/(g_i|_{V_i})$ pour une certaine fonction~$a_{i,\ell}$ sur~$X$ ; l'invisibilité 
de~$h_{i,\ell}$ entraîne celle de~$a_{i,\ell}|_{V_i}$. 

Soit~$E$ l'ensemble des fonctions de la forme~$g_i$, $f_{i,j}$
et~$a_{i,\ell}$ pour~$i,j$ et~$\ell$ variables. Par construction, $\omega$
est G-descriptible au moyen de~$E$, et si c'est une forme elle est même
descriptible au moyen de~$E$. 
\end{proof}

\begin{lemm}\label{isole.bord}
Soit $f\colon Y\to X$ un morphisme entre espaces $k$-affinoïdes et 
soit~$K$ un compact contenu dans $\Int(Y/X)$. Il existe une famille 
$(f_1,\ldots, f_n)$ de fonctions analytiques non nilpotentes sur $Y$ 
telle que les propriétés suivantes soient satisfaites, en désignant 
pour tout $i$ par $\rho_i$ le rayon spectral de $f_i$ : 
\begin{enumerate}
\item Pour tout $y\in K$ et tout $i$ on a $\abs{f_i(y)}<\rho_i$;
\item Pour tout $y\in \partial (Y/X)$ il existe $i$ tel que 
$\abs{f_i(y)}=\rho_i$.
\end{enumerate}
\end{lemm}
\begin{proof}
Notons $\widetilde Y$ et $\widetilde X$ les réductions 
graduées à la Temkin
des espaces affinoïdes~$Y$ et~$X$
(\ref{ss-reduction-temkin})
ainsi que $\tilde f\colon  \widetilde Y \to \widetilde X$ 
l'application induite par $f$. Soit $\pi \colon  Y\to \widetilde Y$ 
la flèche de réduction. Notons $\mathscr A$ et $\mathscr B$ les 
algèbres respectivement associées à $X$ et $Y$, et 
$\widetilde{\mathscr A}$ et $\widetilde{\mathscr B}$ leurs réductions 
graduées. 

Soit ${\bf y}\in \widetilde Y$. Notons $J_{\bf y}$ (resp $I_{\bf y}$) 
l'idéal de $\widetilde{\mathscr B}$ (resp. $\widetilde{\mathscr A}$) 
engendré par les éléments homogènes s'annulant sur $\bf y$ (resp. 
$\tilde f({\bf y})$). 
Soit $E$ le sous-ensemble de $\widetilde 
Y$ formé des points~$\bf y$ tel que $\widetilde{\mathscr B}/J_{\bf 
y}$ soit finie sur $\widetilde{\mathscr A}/I_{\bf y}$. En vertu du 
\emph{going-up} gradué, si ${\bf y}\in E$ alors $\overline{\{\bf 
y\}}\subset E$. 

L'intérieur $\Int (Y/X)$ est précisément égal à $\pi^{-1}(E)$
(\cite{temkin2004}, Proposition~3.1).

\medskip
Soit $y\in K$. Comme $K\subset  \Int(Y/X)$, on a $\pi(y)\in E$. 
L'image réciproque du fermé $\overline{\{\pi(y)\}}$ de $\widetilde Y$ 
est ouverte, et la compacité de $K$ entraîne alors que $\pi(K)$ est 
contenu dans une réunion finie $F=\bigcup \overline {\{{\bf y}_i\}}$, 
où les ${\bf y}_i$ appartiennent à $E$ ; notons que $F$ est un fermé 
de $\widetilde Y$, qui est contenu dans $E$. 

Par quasi-compacité de 
$\widetilde Y$, 
il existe une famille $(f_1,\ldots, f_n)$ de fonctions analytiques 
non nilpotentes sur $Y$, de rayons spectraux respectifs 
$\rho_1,\ldots, \rho_n$, telle que $F$ soit le lieu des zéros des 
$\widetilde{f_i}$, où $\widetilde{f_i}$  désigne l'image canonique de 
$f_i$ dans le groupe des éléments homogènes de degré~$\rho_i$ de 
$\widetilde{\mathscr A}$. 
Démontrons que la famille~$(f_1,\dots,f_n)$ vérifie les
conditions du lemme.

Soit $y\in K$. Puisque $\pi(y)\in F$, on a pour tout $i$ l'égalité 
$\tilde{f_i}(\pi(y))=0$, ce qui signifie exactement que 
$\abs{f_i(y)}<\rho_i$. 

Soit $y\in \partial (Y/X)$. Le point $\pi(y)$ n'appartient pas à $E$, 
et \emph{a fortiori} pas à $F$. Il s'ensuit qu'il existe $i$ tel que 
$\tilde{f_i}(\pi(y))\neq 0$, ce qui signifie exactement que 
$\abs{f_i(y)}$ est égal à~$\rho_i$. 
\end{proof}

\begin{prop}\label{prop.tropicalisation-globale}
Soit $Y\to X$ un morphisme entre espaces $k$-affinoïdes, 
où $Y$ est
irréductible 
de dimension $n$.
Soit $(\omega_1,\ldots, \omega_m)$ une famille finie de G-formes
(resp. formes) à coefficients
mesurables
sur~$Y$. 
Soit~$K$ un compact contenu dans $\Int(Y/X)$ constitué de points~$y$ 
tels que $d_k(y)=n$. 

Il existe 
un domaine rationnel~$V$ de $Y$
qui est un 
voisinage affinoïde de~$K$ dans~$\Int(Y/X)$
et un moment~$g\colon  V\to T$ tels que les deux propriétés suivantes 
soient satisfaites : 
\begin{enumerate}
\item Chacune des $\omega_i|_V$ est G-tropicalisée (resp. est tropicalisée)
par~$g$ ; 

\item On a $g_\trop(\partial(V/Y))\cap g_\trop (K)=\emptyset$.
\end{enumerate}
\end{prop}

\begin{proof}
Soit $\mathscr A$ l'algèbre
des fonctions analytiques sur $Y$. 
Il résulte de la proposition \ref{trop.partiel} qu'il existe un 
ensemble fini $\{h_1,\ldots, h_N\}$ de fonctions appartenant 
à $\mathscr A$ au moyen desquelles chacune des~$\omega_i$
est G-descriptible (resp. descriptible).  
On peut par ailleurs supposer qu'aucune des~$h_j$ n'est
nilpotente (remarque~\ref{descript-fonction-nonnulle}). 
%
%
%

Soit $y\in K$. Comme $d_k(y)=n$,
le point $y$ n'est situé sur aucun fermé de Zariski de $Y$ de 
dimension $<n$ ; comme $Y$ est irréductible de dimension~$n$,
cela signifie que $y$ est Zariski-dense dans~$Y$. Il s'ensuit 
qu'aucune des fonctions~$h_j$ ne s'annule sur~$K$. 


Le compact $K$ est contenu dans $\Int (Y/X)$. 
En vertu du lemme~\ref{isole.bord}, 
il existe un nombre fini de fonctions analytiques  
non nilpotentes $\lambda_1,\ldots, \lambda_r$ sur $Y$, de rayons 
spectraux respectifs $\rho_1,\ldots, \rho_r$, telles que : 
\begin{itemize}
\item
pour tout $\ell$ et tout $y\in K$, 
 $\abs{\lambda_\ell(y)}<\rho_\ell$ ; 

\item pour tout $y\in \partial (Y/X)$, il existe $\ell$ tel que 
$\abs{\lambda_\ell}=\rho_\ell$. 
\end{itemize}

Choisissons pour tout $j$ un nombre réel strictement positif $s_j$ tel que 
$\abs{h_j}>s_j$ sur $K$, et soit $W$ le domaine affinoïde 
rationnel de $Y$ défini 
par la conjonction d'inégalités $\abs{h_j}\geq s_j$. 

Pour tout 
$\ell\in \{1,\ldots, r\}$, la fonction~$\rho_\ell$ est non nilpotente ; 
comme tout point de~$K$ est Zariski-dense dans~$Y$, la
fonction~$\rho_\ell$ ne s'annule pas sur~$K$. Il existe dès lors deux nombres 
réels $r_\ell$ et $R_\ell$ tels que 
$0<r_\ell<R_\ell<\rho_\ell$ et tels que 
$r_\ell<\abs{\lambda_\ell}<R_\ell$ sur $K$. Soit~$V$
le domaine affinoïde de $W$ 
défini par la conjonction des inégalités $r_\ell\leq 
\abs{\lambda_\ell}\leq R_\ell$ pour $\ell\in \{1,\ldots, r\}$ ; 
c'est un domaine rationnel de $W$ et, partant, de~$Y$.

Soit $y\in V$. On a   
$\abs{\lambda_\ell(y)}\leq R_\ell<\rho_\ell$ pour
tout~$\ell$. En conséquence, le 
point $y$ n'appartient pas à $\partial (Y/X)$ ; il vient $V\subset 
\Int (Y/X)$. 

Soit $y\in \partial(V/Y)$. Par définition du domaine affinoïde $V$, on 
est dans l'un des deux cas suivants, non exclusifs l'un de l'autre. 

\begin{itemize}
\item \emph{Premier cas : $y\in \partial (W/Y)$}. 
Il existe alors $j$ tel que $\abs{h_j}(y)=s_j$, 
et $\abs{h_j(y)}\notin \abs{h_j}(K)$. 

\item \emph{Second cas : $y\in \partial (V/W)$.} 
Il existe alors~$\ell$ tel que
$\abs{\lambda_\ell(y)}\in\{r_\ell,R_\ell\}$. 
Par définition de $r_\ell$ et $R_\ell,$, on en 
déduit 
que $\abs{\lambda_\ell(y)}\notin \abs{\lambda_\ell}(K)$.
\end{itemize}

Par construction, les~$h_j$ et les~$\lambda_\ell$
sont inversibles sur~$V$ et définissent donc un moment~$g$ 
sur~$V$. Il résulte de ce qui précède que 
$g_\trop(\partial(V/Y))\cap g_\trop(K)=\emptyset$.

Par ailleurs comme pour tout~$i$,
la (G-)forme~$\omega_i$ est G-descriptible (resp. descriptible)
 au moyen des~$h_j$, il en va de même de sa restriction~$\omega_i|_V$. 
Celle-ci est
 \emph{a fortiori}
G-descriptible (\resp descriptible)
au moyen des~$h_j$ et des~$\lambda_\ell$ ; 
en vertu du lemme~\ref{descript.trop}, 
la forme~$\omega_i|_V$ est alors G-tropicalisée (resp. tropicalisée) par~$g$.
\end{proof}

\section{Harmonie intérieure}

\begin{theo}\label{theo.interieur-harmonieux}
Soit $X$ un espace analytique de niveau~$n$.
L'intérieur de~$X$ est harmonieux.
\end{theo}

\begin{proof}
Soit $U$ un ouvert de~$\Int(X)$, soit $f\colon U\to \R^n$ une carte
tropicale et soit $\mathscr C$ une décomposition cellulaire
de $\Sigma_f^{(n)}$ en $f$-cellules 
qui est adaptée au calibrage de~$\Sigma_f^{(n)}$.
Soit $F$ une $(n-1)$-cellule de~$\mathscr C$; démontrons que $F$
est harmonieuse relativement à~$f$.

Quitte à remplacer~$X$ par~$U$, on peut supposer que $f$
est définie sur~$X$ et que $X$ est sans bord.

Par définition de la structure d'espace affiné sur~$X$,
il existe un moment $g\colon X\to \gm^n$ sur~$X$
et un  isomorphisme affine~$p\colon \R^n\to\R^n$ tels que
$g_\trop=p\circ f$. Les discordance de~$F$
relativement à~$f$ et~$g_\trop$ se déduisent l'une de l'autre par transport
de structure, de sorte que l'on peut supposer que la carte tropicale
donnée est de la forme $f_\trop$, où $f\colon X\to T$ est un moment.

Soit $\mathscr C_{>F}$ l'ensemble des $n$-cellules de~$\mathscr C$
contenant~$F$; il est fini et non vide, car $\Sigma_f^{(n)}$ est purement
de dimension~$n$.
Fixons une orientation~$o$ de~$f_\trop(F)$ et, 
pour tout~$C\in\mathscr C_{>F}$,
notons~$o_C$ l'orientation de $f_\trop(C)$ 
déduite de celle de~$f_\trop(F)$
par la règle de la normale sortante à~$f_\trop(F)$.
Il existe un unique $n$-vecteur, $\tilde\mu_C$, sur 
le sous-espace affine de~$T_\trop$
engendré par~$f_\trop(C)$
tel que $(f_\trop)_*(\mu_C)$ soit la classe du couple $(\tilde\mu_C,o_C)$.

Notons $\tilde\mu_F$
la somme des fonctions~$\tilde\mu_C|_F$
pour~$C$ parcourant~$\mathscr C_{>F}$.
Par définition, la discordance de~$F$ relativement à~$f_\trop$ est 
la classe~$\mu_F$ du couple $(\tilde\mu_F,o)$ 
dans le produit contracté $\Lambda^n\vec E\times^ { \{\pm1\} } \Or(F)$.
Il s'agit de démontrer que $\tilde\mu_F=0$.

Soit $\Omega$ l'ensemble des applications linéaires $u$ de~$T_\trop$
dans~$\R^n$ telles que $u|_{\langle f_\trop(C)\rangle}$
soit injectif pour tout $C\in\mathscr C_{>F}$.
Soit $q\colon T\to\gm^n$ un morphisme de tores tel que $q_\trop\in\Omega$.

Pour tout $C\in\mathscr C_{>F}$, soit $\sigma_{C}\colon\R^n\to T_\trop$
l'unique section affine de~$q_\trop$ 
dont l'image est égale à~$\langle f_\trop(C)\rangle$.
Soit $e=e_1\wedge\dots\wedge e_n$ le $n$-vecteur canonique de~$\R^n$
et soit $\abs e$ le vecteur-volume associé.
L'hyperplan $H_0=\langle (q\circ f)_\trop(F)\rangle$
est naturellement orienté par l'image de l'orientation de~$F$.
Notons $H^+$ le demi-espace bordé par~$H_0$ que la règle de la normale
sortante munit de l'orientation canonique de~$\R^n$; notons~$H^-$ l'autre
demi-espace. 
Pour tout $C\in\mathscr C_{>F}$, la $n$-cellule $(q\circ f)_\trop(C)$
est contenue dans l'un des deux demi-espaces~$H^+$ et~$H^-$, 
et l'un d'eux seulement. 
Posons $\eta_{C}=1$ si $(q\circ f)_\trop(C)\subset H^+$
et $\eta_{C}=-1$ sinon. 

Soit $C\in\mathscr C_{>F}$; 
comme la cellule~$C$ est contenue dans $\Int(X)$, 
c'est une $q\circ f$-cellule commode. 
Il existe donc un entier~$d_{C}>0$ tel que le morphisme $q\circ f$ soit
fini et plat de degré~$d_{C}$ en tout point de~$\mathring C$
(proposition~\ref{prop.commode-fini-plat}).

Par définition du calibrage canonique, on a 
$ \mu_C =(q\circ f)_\trop^* (d_C \abs e)$,
d'où la relation
\[ (q_\trop)_*(\tilde\mu_C)  =\eta_{C} d_{C} e. \]
Le lemme~\ref{lemm.nullite} ci-dessous entraîne alors
que $(q_\trop)_*(\tilde\mu_F)=0$.
Autrement dit, $\tilde\mu_F$ appartient au noyau de l'application
linéaire $\bigwedge^{n}(q_\trop)_*$.

Ceci vaut pour tout morphisme affine~$q$ de tores tel que $q_\trop$
appartienne à~$\Omega$. Cet ensemble~$\Omega$ est un ouvert
de Zariski dense de $\Hom(T_\trop,\R^n)$, donc l'ensemble
$\Omega_\Z$ de ses points entiers est aussi Zariski dense,
et l'intersection des noyaux des applications
linéaires $\bigwedge^n (q_\trop)_*$,
pour $q\in \Hom(T,\gm^n)$ tel que $q_\trop\in\Omega$
est réduite à~$0$.
On en déduit que $\tilde\mu_F=0$,
ainsi qu'il fallait démontrer.
\end{proof}

\begin{lemm}\label{lemm.nullite}
On a $\sum_{C\in\mathscr C_{>F}} \eta_{C} d_{C} =0$.
\end{lemm}
\begin{proof} 
Soit $x$ un point de~$\mathring F$. Par définition, $\Sigma_f^{(n)}$
est purement de dimension~$n$, l'ensemble~$\mathscr C_{>F}$
n'est pas vide, et $q$ a été choisie de sorte que $(q\circ f)_\trop$
soit injective sur tout élément de~$\mathscr C_{>F}$. 
Par suite, $q\circ f$ est de dimension tropicale~$n$ en~$x$, 
si bien que $q\circ f (x)\in S(\gm^n)$.
D'après le lemme~\ref{lemme.finiplat-local}, 
le morphisme $q\circ f$ est fini et plat en~$x$,
car $x$ n'appartient pas au bord de~$X$.

Soit $U$ un voisinage ouvert et connexe de~$x$ dans~$X$
et soit $V$ un voisinage ouvert de $q\circ f(x)$ dans~$\gm^n$ 
tels que $q\circ f$ induise un morphisme fini et plat de~$U$ sur~$V$,
de degré $d_{q\circ f}(x)$. 
En particulier, $x$ est le seul point de sa fibre dans~$U$,
et $q\circ f(x)=(q\circ f)_\trop(x)_\sq$.
On peut en outre supposer que $U\cap\Sigma_f^{(n)}$ 
est contenu dans la réunion des cellules~$C$, pour $C\in\mathscr C_{>F}$.
Notons aussi que $\partial(U)=\emptyset$ car $U$ est un ouvert de~$X$
et $X$ est sans bord.

Soit $\mathscr C_{>F}^+$ l'ensemble des cellules $C\in\mathscr C_{>F}$
telles que $\eta_C=1$. 
Pour tout $C\in \mathscr C_{>F}^+$, $q\circ f(C)$ est un voisinage
de $q\circ f(x)$ dans $H^+_\sq$ 
et l'intersection des cellules ouvertes $q\circ f(\mathring C)$
rencontre~$V$. Choisissons donc un point~$z\in V$ 
tel que $z\in q\circ f(\mathring C)$ pour tout $C\in\mathscr C_{>F}^+$,
et $z\not\in H_{0,\sq}$.

Soit $C\in\mathscr C^+_{>F}$.
Par choix de~$z$, il existe un point $y\in\mathring C$
 tel que $q\circ f(y)=z$.
Comme $q\circ f|_C$ est injective, $y$ est le seul
antécédent de~$z$ sur~$C$; notons-le $y_C$.

Inversement, soit $y$ un point de~$U$ tel que $q\circ f(y)=z$;
d'après la proposition~\ref{prop.trop.local},
on a $y\in \Sigma_f^{(n)}$.
Comme $y\in U$, il existe une cellule $C\in\mathscr C_{>F}$
telle que $y\in C$;
comme $q\circ f(y)\in (H^+\setminus H_0)_\sq$,
on a $C\in\mathscr C_{>F}^+$.
Par suite, $y=y_C$.

Par définition du degré,  on donc 
\[ \sum_{C\in\mathscr C_{>F}^+} d_C 
= \sum_{C\in\mathscr C_{>F}^+} d_{q\circ f}(y_C)
 = d_{q\circ f}(x), \]
puisque $q\circ f$ induit un morphisme fini et plat de degré~$d_{q\circ f}(x)$
de~$V$ sur~$U$.

Notons $\mathscr C_{>F}^-$ l'ensemble des cellules $C\in\mathscr C_{>F}$
telles que $\eta_C=-1$.
Le même argument démontre que 
$\sum_{C\in\mathscr C_{>F}^-} d_C=d_{q\circ f}(x)$.

On a donc
\[ \sum_{C\in\mathscr C_{>F}} \eta_C d_C = 
 \sum_{C\in\mathscr C_{>F}^+}  d_C -  
 \sum_{C\in\mathscr C_{>F}^-}  d_C 
= d_{q\circ f}(x)-d_{q\circ f}(x)=0,
\]  
d'où le lemme.
\end{proof}

\subsection{}
La formule~\eqref{eq.integrale-but} définit explicitement
l'intégrale d'une forme tropicale de type~$(n,n)$ ou~$(n-1,n)$. 
On a démontré au \S\ref{ss.integrale-but-suffit}
que cette relation
dans des cartes tropicales permet de reconstituer
la mesure de Radon associée à une G-forme de type~$(n,n)$.
Démontrons ici qu'il en est  de même pour les formes
de type~$(n-1,n)$, ce qui fournira une définition tropicale
alternative de cette intégrale.

Soit $X$ un espace analytique de niveau~$n$,
soit $\omega$ une forme de type~$(n-1,n)$ sur~$X$,
à support compact et à coefficients localement bornés. 
Supposons aussi que $X$ est bon.

Pour tout point~$x$ du support de~$\omega$, on a $d_k(x)=n$,
donc il existe (lemme~\ref{lemm.voisinage-irreductible})
deux voisinages affinoïdes irréductibles~$U_x$ 
et~$V_x$ de~$x$ dans~$X$
tels que $V_x\subset \Int(U_x/X)$. Notons
que $\Int(V_x/U_x)=\Int(V_x/X)$ est une partie ouverte
de chacun des espaces~$V_x$, $U_x$ et~$X$.
Le support de~$\omega$ 
étant compact, il existe une partie finie~$S$ de~$\supp(\omega)$
telle que la réunion des~$\Int(V_x/X)$, pour $x\in S$, recouvre~$\supp(\omega)$.
Comme $X$ est supposé bon, il est 
localement holomorphiquement séparé
et le corollaire~\ref{prop.berk-partition.1.lisse-compact}
entraîne qu'il existe une famille $(\lambda_x)_{x\in S}$
de fonctions lisses sur~$X$ telle que, pour tout~$x\in S$,
le support de~$\lambda_x$ soit une partie compacte de~$\Int(V_x/X)$
et telle que $\sum_{x\in S} \lambda_x$ soit égale à~$1$
au voisinage de~$\supp(\omega)$.

On a  $\omega=\sum_{x\in S} \lambda_x \omega$, donc 
\[ \int_X^\partial\omega  = \sum_{x\in S} \int_X^\partial \lambda_x \omega. \]
Fixons $x\in S$.  Soit~$K_x$ le support  de~$\lambda_x \omega$; 
c'est une partie compacte de~$\Int(V_x/X)$
et l'on a $d_k(y)=n$ pour tout $y\in K_x$.
D'après la proposition~\ref{prop.tropicalisation-globale},
appliquée à l'injection de~$V_x$ dans~$U_x$,
il existe un voisinage affinoïde~$W_x$ de~$K_x$ dans~$\Int(V_x/X)$
sur lequel $\lambda_x\omega$ est tropicale.
Comme $\Int(W_x/X)$ est une partie ouverte de~$X$ qui contient
le support de~$\lambda_x\omega$, on a 
\[ \int_X^\partial \lambda_x\omega =  \int^\partial_{\Int(W_x/X)} \lambda_x\omega. \]
Comme $\Int(W_x/X)$ est une partie ouverte de~$W_x$ qui contient
le support de~$\lambda_x\omega$, on a  aussi
\[ \int^\partial_{W_x} \lambda_x\omega =  \int^\partial_{\Int(W_x/X)} \lambda_x\omega. \]
On a donc 
\[ \int_X^\partial \lambda_x\omega =  \int^\partial_{W_x} \lambda_x\omega, \]
et cette dernière intégrale peut se calculer à l'aide de la formule~\eqref{eq.integrale-but}.

A priori, le calcul tropical
de l'intégrale $\int^\partial_{W_x} \lambda_x\omega$ 
fait intervenir toute la tropicalisation de~$\partial(W_x)$.
Remarquons cependant que la pleine force de la proposition~\ref{prop.tropicalisation-globale} démontre qu'on peut choisir~$W_x$ et un moment
tropicalisant~$\omega$ de sorte que ce calcul
ne fasse intervenir que la tropicalisation de~$\partial(X)\cap W_x$.

\begin{rema}
Soit $X$ un espace analytique de niveau~$n$. Soit $f$ un moment
sur~$X$ et soit $C$ une $f$-cellule commode de dimension~$n$.
Alors $\mathring C$ est un ouvert de~$\Sigma_f^{(n)}$,
par définition d'une cellule commode;
comme il ne rencontre le bord de~$X$,
il est harmonieux relativement à~$f$ (théorème~\ref{theo.interieur-harmonieux}).
D'après la remarque~\ref{rema.harmonieux-calibrage-constant},
son calibrage est constant relativement à~$f$.

Mais cette propriété résulte directement de la construction du calibrage
de~$C$, considérée comme partie squelettique de~$\Sigma_f$,
au moyen d'une décomposition cellulaire de~$C$ en $f$-cellules commodes 
dont~$C$ elle-même
est l'unique cellule de dimension~$n$.
\end{rema}

\section{L'espace tropical associé à un espace analytique}

Soit $X$ un espace analytique purement de dimension~$n$, 
topologiquement séparé. On le considère comme espace G-tropical
calibré de niveau~$n$.
Il est riche (théorème~\ref{theo.riche}).

\begin{theo}\label{theo.an-tropical}
Soit $X$ un espace analytique purement de dimension~$n$,
topologiquement séparé.
Muni de son bord $\partial(X)$  au sens de la théorie des espaces analytiques,
l'espace G-tropical affiné et calibré de niveau~$n$ associé à~$X$
est un espace tropical de niveau~$n$.
\end{theo}
\begin{proof}
Vérifions les conditions de la définition~\ref{defi.esp-tropical}.

L'intérieur $\Int(X)$ de~$X$ est 
localement holomorphiquement séparé  (\ref{hol-sep});
il est donc localement différentiellement séparé 
(corollaire~\ref{coro.berk-local-diff-sep}).
Il est décent (proposition~\ref{prop.presque-bon-decent})
car il est sans bord, donc bon, donc presque bon.
Il est très riche (théorème~\ref{theo.berk-tres-riche})
et harmonieux (théorème~\ref{theo.interieur-harmonieux}).

Et il résulte du corollaire \ref{coro-bor-petit}
que $\Sigma\cap \partial X$ est $n$-négligeable pour toute partie paralinéaire $\Sigma$ 
de $X$, de même que $\Sigma \cap \partial (V/X)$ pour tout domaine $V$ de $X$ car
$\partial(V/X)\subset \partial V$. 
\end{proof}

On munit ainsi tout espace analytique~$X$, topologiquement séparé,
et purement de dimension~$n$, 
d'une structure d'espace tropical de niveau~$n$.

\subsection{}
Soit $X$ un espace analytique topologiquement séparé, purement
de dimension~$n$ et soit $V$ un domaine de~$X$.
Le domaine~$V$ est muni de deux façons, a priori distinctes,
d'une structure d'espace tropical de niveau~$n$,
soit comme espace analytique, soit comme domaine de l'espace
tropical~$X$.
La formule $\partial(V)=(\partial(X)\cap V) \cup \partial(V/X) $
montre que ces deux structures coïncident.

\begin{rema}
Soit $X$ un espace analytique de niveau~$n$; pour que l'on puisse
le munir d'une structure d'espace tropical de niveau~$n$,
il faut et il suffit que $X$ soit purement de dimension~$n$.
En effet, dans le cas contraire, 
il possède un domaine
analytique compact non vide~$V$ tel que $\dim(V)<n$.
L'espace~$X$ n'est donc pas riche: la tropicalisation constante de~$V$
n'est couverte par aucun squelette de~$V$ puisque la seule partie
squelettique de~$V$ est vide.
\end{rema}

\chapter{Courants sur les espaces analytiques}

Soit $X$ un espace $k$-analytique topologiquement
séparé et purement de dimension~$n$. Pour alléger la terminologie,
nous dirons plus simplement que $X$ est un $n$-espace $k$-analytique.\index{espace $k$-analytique!$n$-espace $k$-analytique}
Nous l'avons muni d'une structure d'espace G-tropical de niveau~$n$,
calibré et affiné, et démontré (théorème~\ref{theo.an-tropical}) que
muni de son bord au sens de la théorie des espaces $k$-analytiques,
c'est un espace tropical de niveau~$n$.

Il en résulte des notions de formes et G-formes différentielles
de type~$(p,q)$ sur~$X$, à coefficients dans une classe de coefficients,
une théorie de l'intégration des $(n,n)$-formes et des $(n-1,n)$-formes
ainsi qu'une théorie de courants.

Dans ce chapitre, nous explorons plus avant les propriétés
spécifiques des courants sur les espaces analytiques
qui semblent difficilement formulables dans le contexte des espaces tropicaux.

\section{Quelques formules}

\begin{prop}\label{prop.image-directe-chgt-base}
Soit $X$ un $n$-un espace $k$-analytique,
soit $K$ une extension complète de~$k$
et soit $p\colon X_K \to X$ le morphisme de changement de base.
On a les égalités de courants $p_*(\delta_{X_K})=\delta_X$
et $p_*(\delta_{X_K}^\partial)=\delta_X^\partial$.
\end{prop}
\begin{proof}
Pour la première relation, 
il s'agit de démontrer que pour toute forme lisse~$\omega$ 
de type~$(n,n)$,
à support compact, on a 
$ \int_{X_K} p^*\omega = \int_X \omega$;
c'est la formule~\ref{prop.mesure-chgt-base}.

La seconde se prouve de même en utilisant la formule~\ref{prop.mesure-chgt-base-bord}.
\end{proof}
 
\begin{prop}
Soit $f\colon Y\to X$ un morphisme compact de $n$-espaces $k$-analytiques.

\begin{enumerate}
\item 
Le degré calibré, plus précisément tout prolongement à~$X$,
est une fonction localement intégrable sur~$X$.

\item
On a les égalités de courants $f_*\delta_Y = \degcal_f \delta_X$
et $f_*\delta_Y^\partial=\degcal_f \delta_X^\partial$.

\item
Soit $\omega$ une G-forme lisse sur~$X$ 
et soit~$u$ une fonction localement intégrable sur~$X$, 
à valeurs dans~$\R\cup\{-\infty\}$ (\cf \S\ref{ss.G-forme-L1loc}).
Si $f^*u$ est localement intégrable sur~$Y$,
on a $f_*(f^*u\cdot f^*[\omega])=\degcal_f u\cdot [\omega]$.
\end{enumerate}
\end{prop}
Notons que pour \emph b), l'assertion que $f^*u$ est localement
intégrable est automatique si $u$ est localement bornée,
ou si $f$ ne contracte aucune composante de dimension~$n$
(car alors $f$ induit une immersion par morceaux
entre parties squelettiques).
\begin{proof}
\begin{enumerate}
\item
Il s'agit de démontrer que pour toute forme~$\omega$
à support compact sur~$X$, de type~$(n,n)$ ou $(n-1,n)$, 
on a $ \int_{Y} f^*\omega = \int_X \degcal_f\omega$,
respectivement
$ \int^\partial_{Y} f^*\omega = \int^\partial_X \degcal_f\omega$.
Cela résulte de la proposition~\ref{prop.mesure-f*-degre-integral}.

\item
La démonstration est analogue.

\item
Notons $(p,q)$ le degré de~$\omega$. Soit $\alpha$ une forme
lisse à support compact sur~$X$ de bidegré~$(n-p,n-q)$.
Par définition, on a
\begin{align*}
 \langle f_*(f^*u\cdot f^*[\omega]),\alpha\rangle 
&= \langle f^*u\cdot f^*[\omega], f^*\alpha\rangle \\
& = \int_Y f^*u f^*[\omega\wedge\alpha] \\
& = \int_Y f^* u (f^* (\omega\wedge\alpha))_\Rad \\
& = \int_X u f_* (f^* (\omega\wedge\alpha))_\Rad \\
& = \int_X u \degcal_f (\omega\wedge\alpha)_\Rad \\
& = \langle u \degcal_f \delta_X [\omega], \alpha\rangle, 
\end{align*}
compte tenu de la proposition~\ref{prop.mesure-f*-degre-integral}.
\qedhere
\end{enumerate}
\end{proof}

\begin{prop}
Soit $f\colon Y\to X$ un morphisme propre
de $n$-espaces $k$-analytiques sans bord.
On suppose que $f$ est fini au-dessus du complémentaire
d'un fermé de Zariski de dimension~$<n$
et que la fonction $\degcal_f$ est constante
de valeur~$d$ sur~$X^\gen$.
Soit $u$ une fonction localement intégrable sur~$X$;
si $f^*u$ est localement intégrable, on a
\[ f_*(\ddc f^*u) = d \cdot \ddc u. \]
\end{prop}
\begin{proof}
D'après la proposition précédente, on a
\[ f_*(\ddc [f^*u]) = \ddc(f_*[f^*u])= \ddc (\degcal_f [u]). \]
Puisque $\degcal_f$ est constant de valeur~$d$ sur~$X^\gen$, on peut écrire
\[ \ddc (\degcal_f [u])= \ddc (d\cdot [u])=d\cdot \ddc [u], \]
ce qu'il fallait démontrer.
\end{proof}
\begin{exem}\label{exem.f*ddcf*u}
Plaçons-nous sous les hypothèses de la proposition précédente.
Soit $h$ une fonction méromorphe sur~$X$ dont le lieu
de non-inversibilité est de dimension~$<n$; supposons que
$f^*h$ soit une fonction méromorphe régulière sur~$Y$.
Alors, $f^*\log\abs{h}$ est localement intégrable sur~$Y$.
La formule de Poincaré--Lelong (théorème~\ref{theo.poincare-lelong})
affirme que 
$\ddc[f^*\log\abs h]=\delta_{\div(f^*h)}$.
On a donc
$ f_* \delta_{\div(f^*h)}= d \cdot \ddc [\log \abs h]$.

Donnons deux cas particuliers.

\begin{enumerate}
\item
Supposons que $h$ soit une fonction méromorphe régulière sur~$X$.
On en déduit alors $f_* \delta_{\div(f^*h)}=d \delta_{\div(h)}$.
Si $X$ est réduit, le morphisme~$f$ est fini et plat
au-dessus de~$X^\gen$ de degré constant~$d$.
Cela reflète alors la formule $f_*[\div(f^*h)]=d\cdot [\div(h)]$
en théorie de l'intersection.

\item
Soit $U$ l'ouvert complémentaire de la réunion des composantes
immergées de~$X$, soit~$Y$ son adhérence analytique,
c'est-à-dire le plus petit sous-espace analytique fermé de~$X$
qui contient~$U$
(\cite{Ducros-2021}, définition~2.17). 
L'immersion fermée $f\colon Y\to X$ est nilpotente
et est un isomorphisme au-dessus de~$U$; on peut donc prendre $d=1$.
Pour toute fonction localement intégrable~$u$ sur~$X$,
la fonction~$f^*u$ sur~$Y$ est localement intégrable et l'on a
$f_*(\ddc [f^*u])=\ddc [u]$.

Si le lieu d'inversibilité de~$h$ est dense, la fonction méromorphe~$f^*h$
est régulière et justiciable de la formule de Poincaré--Lelong
du théorème~\ref{theo.poincare-lelong},
alors que $h$ pourrait ne pas l'être
et on obtient $\ddc \log \abs h = f_* \delta_{\div(f^*h)}$.

Pour un exemple concret de cette situation, supposons que~$X$
soit l'analytifié de $\Spec(k[S,T])/(T^2, ST)$;
c'est une droite affine, réduite hors de l'origine~$0$,
de sorte que $Y$ est l'analytifié de~$\Spec(k[S])$.
La fonction holomorphe~$S$ n'est pas régulière sur~$X$,
en raison de l'origine, mais sa restriction à~$Y$ l'est.
On obtient $\ddc \log \abs S=\delta_0$,
bien que le sous-espace analytique fermé de~$X$ défini par~$S$,
égal à $\mathscr M(k[T]/(T^2))$, soit de multiplicité~2.
\end{enumerate}
\end{exem}

\begin{prop}\label{prop.image-directe-fini}
Soit $f\colon Y\to X$ un morphisme propre
de $n$-espaces $k$-analytiques sans bord.
On suppose que $f$ est fini au-dessus du complémentaire
d'un fermé de Zariski de dimension~$<n$
et que la fonction $ \degcal_f$ est constante
de valeur~$d$ sur~$X^\gen$.

Soit $u_1,\dots,u_m$ des fonctions continues d-psh sur~$X$.
On a 
\[ f_*(\ddc f^* u_1\wedge\cdots\wedge \ddc f^*u_m)
    = d \cdot \ddc u_1\wedge \cdots\wedge\ddc u_m.\]
\end{prop}
\begin{proof}
L'assertion est locale sur~$X$;
par multilinéarité, il suffit alors de traiter où les fonctions~$u_i$ sont psh.
On raisonne alors par récurrence sur~$m$. 
Le cas $m=0$ est l'assertion~\emph a).
Supposons $m\geq 1$ et l'assertion vraie pour~$m-1$. Compte
de la définition du faisceau des fonctions psh, il suffit
de vérifier les deux assertions:
\begin{enumerate}\def\theenumii{\roman{enumii}}\def\labelenumii{(\theenumii)}
\item La formule est vraie si $u_1$ est lisse;
\item Le faisceau des fonctions psh continues
pour lesquelles cette formule vaut est stable par limite uniforme.
\end{enumerate}
Cette dernière assertion découle de la proposition~\ref{prop.BT-continu};
démontrons~\emph a).

Supposons que $u_1$ est une fonction lisse.
Soit $\omega$ une forme lisse à support compact de type~$(n-m,n-m)$ sur~$X$.
On a 
\begin{align*} 
\langle f_*(\ddc f^* u_1\wedge\dots\wedge \ddc f^*u_m), \omega\rangle
\hskip -4cm \\
& = \langle \ddc f^*u_1  \wedge \ddc f^*u_2\wedge \dots \ddc f^*u_m,
 f^*\omega\rangle \\
& =\langle \ddc (f^*u_1  (\ddc u_2\wedge \cdots \wedge\ddc f^*u_m)), f^*\omega \rangle\\
&= \langle f^*u_1 (\ddc f^*u_2\wedge \cdots \wedge \ddc f^*u_m), \ddc(f^*\omega) \rangle\\
&= \langle \ddc f^*u_2\wedge \cdots \wedge \ddc f^*u_m, f^* u_1 \ddc(f^*\omega) \rangle\\
&= \langle \ddc f^*u_2\wedge \dots \ddc f^*u_m, f^* (u_1 \ddc\omega) \rangle\\
&= \langle f_*(\ddc f^*u_2\wedge \cdots \wedge \ddc f^*u_m), u_1 \ddc\omega \rangle.
\end{align*}
Appliquons maintenant l'hypothèse de récurrence; on en déduit
\begin{align*}
 \langle f_*(\ddc f^* u_1\wedge\dots\wedge \ddc f^*u_m), \omega\rangle
 \hskip -4cm \\
& = \langle d \cdot \ddc u_2\wedge \dots \ddc u_m, u_1 \ddc\omega \rangle. \\
& = d \langle u_1 \ddc u_2\wedge \dots \ddc u_m, \ddc\omega \rangle. \\
& = d \langle \ddc (u_1 \ddc u_2\wedge \dots \ddc u_m),\omega \rangle.
\end{align*}
Autrement dit,
\[ f_*(\ddc f^* u_1\wedge\dots\wedge \ddc f^*u_m)
 = d \ddc ( u_1\ddc u_2\wedge \dots \ddc u_m), \]
ainsi qu'il fallait démontrer.
\end{proof}

\section{Continuité des intégrales fibre : cas des squelettes}

\subsection{}
Dans tout cette section, on considère 
un $m$-espace $k$-analytique~$X$,
un $n$-espace $k$-analytique~$Y$ 
et un morphisme $p\colon Y\to X$. 
On pose $d=n-m$; on va étudier le comportement
de l'intégrale d'une $(d,d)$-forme sur~$Y$ dans les fibres de~$p$
qui sont de dimension~$d$. Ces fibres seront considérées
comme espaces analytiques de niveau~$d$.

\begin{prop}\label{prop.variation-cellule}
Soit $X$ et $Y$ deux espaces $k$-analytiques paracompacts, 
purement de dimensions respectives $m$ et $n$. Posons $d=n-m$. 
Soit $p$ un morphisme de $Y$ vers $X$. 
Soit $\mathscr B$ une classe de coefficients qui est ou bien la classe
des fonctions boréliennes localement bornées, ou bien la classe
$\mathscr C^r$ pour un certain $r\in \N\cup\{\infty\}$. 
Soit $\omega$ une $(d,d)$-forme sur $Y$ et soit $\lambda$ une fonction
numérique sur $Y$, toutes deux G-tropicalement de classe $\mathscr B$, 
telles que le support de $\lambda\omega$ soit relativement propre sur $X$. 
Supposons qu'il existe un moment $f\colon X\to\gm^m$
et un moment $g\colon Y\to \gm^d$ 
tels que
$\omega$ soit G-localement tropicalisée par le moment~$(f\circ p,g)$.
Pour tout $x\in\Sigma_f$, la fibre~$Y_x$ est purement de dimension~$d$;
on pose $u(x)= \int_{Y_x}\lambda \omega$. 

\begin{enumerate}
\item
Il existe une décomposition cellulaire~$\mathscr C$ de~$\Sigma_f$
et, pour toute cellule $C\in\mathscr C$, une fonction~$u_C$
de classe $\mathscr B$ sur~$C$
telle que $u(x)= u_C(x)$ pour tout $x\in\mathring C$.

\item Si $\mathscr B\in\{\mathscr C^r\}_{0\leq r\leq \infty}$, si $p$
est séparé et si
$\omega$ est
faiblement positive, alors $u$ est semi-continue supérieurement.

\item
Si  $\mathscr B\in\{\mathscr C^r\}_{0\leq r\leq \infty}$ et si $p\colon Y\to X$ est
sans bord la fonction $u|_{\Sigma_f}$ est G-localement de classe
$\mathscr B$. 
\end{enumerate}
\end{prop}
\begin{proof}
Puisqu'il ne s'agit que du comportement d'intégrales dans les fibres
de~$p$, l'énoncé est inchangé lorsqu'on remplace~$X$ par~$X_\red$,
ce qui permet de supposer que $X$ est réduit.
Le morphisme~$p$ est alors plat au-dessus de tout point de~$\Sigma_f$ 
(théorème~10.3.7 de~\cite{ducros2018}) ;
quitte à remplacer~$X$ par un voisinage analytique de~$\Sigma_f$,
on peut supposer que $p$ est plat. Ses fibres sont alors toutes
purement de dimension~$d$.

Soit $h\colon Y\to \gm^n$ le moment $(f\circ p,g)$. Soient 
$p_1$ et $p_2$ les projections de $X\times \gm^d$ vers
$X$ et  $\gm^d$ respectivement. 
Soit $f'$ le moment $(f\circ p_1,p_2)$ de $X\times \gm^d$ vers $\gm^n$ 
et soit $g'$
le morphisme $(p,g)$ de $Y$ vers $X\times\gm^d$.
Le  lemme \ref{lemm.dimensions-tropicales-fibre} (a)
assure qu'un point $y$ de $Y$ d'image $x$ sur $X$
appartient 
 à $\Sigma_h$ si et seulement si $x$ appartient à $\Sigma_f$
 et $y$ appartient à $\Sigma_{g|_{Y_x}}$. Il assure aussi, lorsqu'on l'applique au 
 morphisme $p_2$ et aux moments $f', f\circ p_1$ et $p_2$ sur l'espace 
 $X\times \gm^d$,  qu'un point $z$ de $X\times \gm^d$ de projection 
 $x$ sur $X$ appartient à $\Sigma_{f'}$ si et seulement si $x$ appartient à 
 $\Sigma_f$ et $z$ appartient au squelette canonique
 $S(\gm[\hr x]^d)$. Par conséquent, $(p_1, p_{2,\trop})$ identifie $\Sigma_{f'}$
 à $\Sigma_f\times \R^d$. De plus, puisque 
 $\Sigma_{g|_{Y_x}}=(g|_{Y_x})^{-1}(S(\gm[\hr x]^d))$, il vient
 $\Sigma_h=(g')^{-1}(\Sigma_{f'})$. La restriction de $g'$ induit 
 donc une application paralinéaire de $\Sigma_h$ dans $\Sigma_{f'}$. 
 Comme $f'\circ g'=h$, cette application est injective sur toute $h$-cellule
 de $\Sigma_h$.

Choisissons une décomposition cellulaire~$\mathscr C$
de~$\Sigma_f$ en $f$-cellules telle que la fonction $x\mapsto \deg^x (f)$
est constante sur chaque cellule ouverte (proposition~\ref{prop.dec-cell-deg}).
Choisissons de même une décomposition cellulaire~$\mathscr D$
de~$\Sigma_h$ en $h$-cellules telle que la fonction $y\mapsto \deg^y(h)$
est constante sur chaque cellule ouverte. On suppose
également, ce qui est loisible, que pour toute cellule~$D$ de~$\mathscr D$,
la forme~$\omega|_D$ est tropicalisée par~$h$
et que la fonction $\lambda|_D$ est de classe~$\mathscr B$.
Quitte à appliquer le corollaire~\ref{prop.decomp-relative},
on peut supposer que le couple $(\mathscr D,\mathscr C)$ est adapté
à l'application paralinéaire de~$\Sigma_h$ dans~$\Sigma_f$ déduite de~$p$.

Pour tout $C\in\mathscr C$, notons $\mathscr D_C$ l'ensemble des cellules~$D\in\mathscr D$ telles que $p(D)=C$.
Pour tout $x\in \mathring C$, la famille des $D\cap Y_x$, 
pour $D\in\mathscr D_C$,
est une décomposition cellulaire de~$\Sigma_{h}\cap {Y_x}$
et
l'on a $\mathring D\cap Y_x=(D\cap Y_x)^\circ$ pour toute $D\in \mathscr D_C$ et tout
$x\in \mathring C$
(corollaire~\ref{coro.dec-cellulaire-fibre}). De plus, pour toute cellule
$D$ de $\mathscr D_C$ et tout~$x\in\mathring C$, la dimension de~$D\cap Y_x$
est égale à~$\dim(D)-\dim(C)$ ; elle est en particulier indépendante de~$x$.
Nous noterons $\mathscr D'_C$ l'ensemble des cellules de $\mathscr D_C$ de dimension 
$\dim C+d$. 

Soit $C$ une cellule de~$\mathscr C$.
Pour toute cellule~$D\in\mathscr D$,
la fonction $y\mapsto \deg^y(h)$ est constante sur~$\mathring D$;
notons~$e_D$ sa valeur.
Pour toute cellule~$C\in\mathscr C$,
la fonction~$x\mapsto \deg^x(f)$ est constante sur~$\mathring C$;
notons~$d_C$ sa valeur.
Puisque $p$ est plat, pour tout $x\in\mathring C$
et tout $D\in\mathscr D_C$, la fonction $y\mapsto \deg^y(g|_{Y_x})$
sur~$\mathring D\cap Y_x$ est constante, de valeur~$e_D/d_C$
(lemme~\ref{lemm.dimensions-tropicales-fibre}, \emph d)).

Soit $D\in\mathscr D$.
Par hypothèse, $\omega|_D$ est tropicalisée par~$h$;
soit alors $\alpha_D$ une $(d,d)$-forme sur~$h_\trop(D)$
telle que $\omega|_D=h^*\alpha_D$. 
Soit $\lambda_D$ une fonction sur~$h_\trop(D)$
telle que $\lambda|_D=h^*\lambda_D$; 
elle est de classe $\mathscr B$. 
Définissons alors une fonction~$v_D$ sur la cellule~$p(D)$ de~$\mathscr C$
par
\[ v_D(x) = \int_{g_\trop(D_x)} \lambda_D \alpha_D, \]
l'intégrale de cette $(d,d)$-forme sur la cellule $g_\trop(D_x)$
de~$\R^d$ étant prise par rapport au calibrage canonique de~$\R^d$.
Puisque $g_\trop(D_x)=g'_\trop(D)_x$, la fonction~$v_D$ est G-localement de classe~$\mathscr B$ sur~$p(D)$
(proposition~\ref{prop.continuite-integrale}).

Soit $C$ une cellule de~$\mathscr C$.
Pour tout $x\in\mathring C$. Les cellules $D_x$ pour $D$ parcourant $\mathscr D_C$ forment une décomposition cellulaire
de $(\Sigma_h)_x$. On a donc 
\begin{align*}
 u(x) &= \int_{Y_x} \lambda\omega
 = \sum_{D \in\mathscr D'_C}  \int_{D_x} \lambda\omega\\
 & = \sum_{D \in\mathscr D'_C }  (e_D/d_C) \int_{g_\trop(D_x)} \lambda_D\omega_D 
 = \sum_{D \in\mathscr D'_C}  (e_D/d_C) v_D(x).
\end{align*}
La fonction~$u_C$ sur~$C$ définie par
\[ u_C (x) 
 = \sum_{D\in \mathscr D'_C }  (e_D/d_C) v_D(x) \]
est G-localement de classe~$\mathscr B$ et
coïncide avec~$u$ sur~$\mathring C$.

Quitte à raffiner la décomposition cellulaire~$\mathscr C$,
l'assertion~\emph a) s'en déduit.

Supposons à partir de maintenant que $\mathscr B$
est la classe $\mathscr C^r$ pour un certain
$r\in \N\cup\{\infty\}$. 
Choisissons une décomposition cellulaire~$\mathscr C$ 
et des fonctions~$u_C$ comme dans la preuve de \emph a). 

Soit $x$ un point de~$\Sigma_f$
et soit $C$ l'unique cellule de~$\mathscr C$ telle que $x\in\mathring C$.
Soit $\Gamma$ une cellule de $\mathscr C$ contenant $C$. 
Nous allons déterminer
la limite de $u(\xi)$ lorsque $\xi$ est un point
de $\mathring \Gamma$ tendant vers $x$, et la comparer avec $u(\xi)$. 

%

Soit donc $\xi$ un point de $\mathring \Gamma$. 
On a 
\[ u(\xi) = u_\Gamma(\xi) = \sum_{\Delta \in\mathscr D'_\Gamma } 
 (e_\Delta/d_\Gamma) v_\Delta(\xi), \]
tandis que 
\[u (x) = u_C(x) =
  \sum_{D \in\mathscr D'_C } 
    (e_D /d_C) v_D(x). \]

Soit $\Delta$ une cellule appartenant à 
$\mathscr D'_\Gamma$. La fonction $v_\Delta$ étant continue, 
$v_\Delta(\xi)$ tend vers $v_\Delta(x)$ lorsque $\xi$ tend vers $x$. 
La fibre $\Delta_x$ est la réunion des cellules $D_x$ pour $D$ parcourant 
$\mathscr D_C$ telles que $D_x\subset \Delta_x$ 
(proposition \ref{coro.dec-cellulaire-fibre} b)). 
Par définition de $\mathscr D_C$, et comme $x\in \mathring C$, une cellule $D$
de $\mathscr D$ appartient à $\mathscr D_C$ si et seulement si $\mathring D_x$ est non vide.
Par conséquent, si une cellule de $D$ de $\mathscr D_C$ est telle que $D_x\subset \Delta_x$ alors
$\Delta$ rencontre $\mathring D$ si bien que $\Delta$ contient $D$.  
Inversement si $D\subset \Delta$ alors $D_x\subset \Delta_x$. 
Compte tenu de la définition de $v_\Delta$, il en résulte que 
\[ v_{\Delta}(\xi) \to \sum_{\substack{D\in \mathscr D'_C \\D\subset \Delta}} v_D(x) \]
lorsque $\xi\to x$.

Par suite, toujours lorsque $\xi\to x$, on a
\begin{align*}
u(\xi)&= \sum_{\Delta \in\mathscr D'_\Gamma } (e_\Delta/d_\Gamma) v_\Delta(\xi)\\
&\to \sum_{\Delta \in\mathscr D'_\Gamma } (e_\Delta/d_\Gamma) \sum_{\substack{D\in \mathscr D'_C \\D\subset \Delta}} v_D(x)\\
&=\sum_{D\in \mathscr D'_C}\sum_{\substack{\Delta\in \mathscr D'_\Gamma\\D\subset \Delta}} (e_\Delta/d_\Gamma) v_D(x)\\
&=\sum_{D\in \mathscr D'_C}v_D(x)\sum_{\substack{\Delta\in \mathscr D'_\Gamma\\D\subset \Delta}} (e_\Delta/d_\Gamma).
\end{align*}
Justifions la dernière égalité. Si $\Delta$ est une cellule de $\mathscr D_\Gamma$ contenant $D$ alors comme la projection de $D$ vers $C$
contient une fibre de dimension $d$, on a nécessairement $\dim \Delta=\dim \Gamma+d$, c'est-à-dire $\Delta\in \mathscr D'_\Gamma$.

Pour tout $x$ dans $X$ et tout $t\in \R^d$, nous noterons $(x,t)$ le point $\sq_x(t)$ de $X\times \gm^d$. 
C'est le point $\eta_t$ de la fibre $\gm[\hr x]^d$ de la projection de $X\times \gm^d$ sur $X$.

\begin{lemm}
Soit $D$ une cellule de $\mathscr D'_C$ et soit $y$ 
un point de l'intérieur de la cellule $D_x$. Posons $t=g_\trop(y)$. 
On a 
\[\sum_{\substack{\Delta \in \mathscr D'_\Gamma\\ D\subset \Delta} } e_\Delta/d_\Gamma=\deg^{(y)}_{(\xi,t)}(g')\]
pour tout $\xi \in \mathring \Gamma$ suffisamment proche de $x$. 
\end{lemm}

Rappelons que $\deg^{(y)}(g')$ est un germe de fonction au voisinage de $(x,t)$, défini en 
\ref{def-degre-bizarre}. 
\begin{proof}
La réunion~$E$ des cellules de~$\mathscr D$ qui contiennent le point~$y$ 
est un voisinage paralinéaire compact de~$y$ dans~$\Sigma_h$. 
Il existe donc un voisinage analytique compact~$V$
de~$y$ dans~$Y$ tel que $V\cap \Sigma_h=E$
(lemme \ref{lemme-proprietes-paralineairesX} (f)). 
Vérifions que $y$ est le seul antécédent de $(x,t)$ par~$g'|_V$.
Comme $(g')^{-1}(\Sigma_{f'})=\Sigma_h$, 
tout antécédent de~$x$ sur~$Y$ est situé sur~$\Sigma_h$. Il suffit
donc de s'assurer que~$y$ est le seul antécédent de $x$ sur~$E$. 
Or $E$ est réunion de $h$-cellules contenant~$y$, 
et la restriction de~$g'$ à chacune de ces cellules est injective, 
d'où l'assertion. 

Soit $E'$ la réunion des $\mathring \Delta$ où $\Delta$ parcourt l'ensemble des cellules
de $\mathscr D$ contenant $y$. C'est un sous-ensemble de $E$ dont le complémentaire
dans $E$ est la réunion des cellules de $\mathscr D$ contenues dans $E$ et ne contenant pas
$y$. Par conséquent, $E'$ est un ouvert de $E$. Il existe donc un ouvert $V'$ de $V$
tel que $V'\cap E=E'$ (et on alors $V'\cap \Sigma_h=E'$). Le point $y$ appartient à $E'$ par
construction, si bien que $V'$ est un voisinage ouvert de $y$ dans $V$. 

Observons que $y\in \mathring D$ car $y$ appartient à l'intérieur de $D_x$ et car $D$ appartient 
à $\mathscr D_C$. Par conséquent, pour toute cellule $\Delta$ de $\mathscr D$ le point $y$ appartient à $\Delta$
si et seulement si $\Delta$ contient $D$. 

Par compacité de $V$, et compte tenu du fait que $(g')^{-1}(x,t)\cap V=\{y\}$, il existe un voisinage ouvert $U$ de $x$
tel que $(g')^{-1}(\xi,t)\cap V$ soit contenu dans $V'$ pour tout $\xi$ appartenant à $U$. 

Soit $\xi$ un point de $\mathring \Gamma\cap U$. 
Par définition, $\deg_{(\xi,t)}(g'|_V)$ est égal à 
la somme des $\deg^{\eta}(g'|_V)$ 
où $\eta$ parcourt l'ensemble des antécédents de~$(\xi,t)$ dans~$V$. 
Nous allons montrer que tout antécédent $\eta$ de $(\xi,t)$
dans $V$ appartient 
à $\bigcup_{\Delta \in \mathscr D'_\Gamma} \Delta$, 
qu'il y a au plus un tel antécédent sur $\Delta$ 
pour
une cellule $\Delta$ donnée appartenant à $\mathscr D'_\Gamma$,
et qu'il appartient le cas échéant à $\mathring \Delta$. 

Soit $\eta$ un point
de~$V$ tel que $g'(\eta)=(\xi,t)$.  
Comme $(g')^{-1}(\Sigma_{f'})=\Sigma_h$, 
le point~$\eta$ appartient à~$\Sigma_h$ ; puisque
$E=V\cap\Sigma_h$, on a $\eta\in E$. 
Le point $\xi$ appartenant à $U$, le point $\eta$ est situé sur $V'$, 
et donc sur $E'$  car $V'\cap \Sigma_h=E'$. 
Il existe en conséquence une cellule $\Delta$ contenant $y$
telle que 
$\eta\in \mathring \Delta$. 
La cellule ouverte $\mathring \Delta$ rencontre la fibre
$Y_\xi$ (elle contient $\eta$) ; d'après le corollaire~\ref{coro.dec-cellulaire-fibre} c)
on a donc $p(\Delta)=\Gamma$, 
autrement dit, $\Delta\in \mathscr D_\Gamma$.
Comme $\Delta$ contient $D$, la projection de $\Delta$ vers $\Gamma$ 
contient une fibre de dimension $d$, si bien que $\dim \Delta=\dim \Gamma+d$, c'est-à-dire $\Delta\in \mathscr D'_\Gamma$. 
Puisque $h|_\Delta$ est injective, $\eta$ est le seul antécédent de $(\xi,t)$ sur $\Delta$. 
Soit $\Delta'$ une cellule de $\mathscr D'_\Gamma$ contenant $\eta$. 
La cellule $\Delta'$ rencontre alors $\mathring \Delta$, donc contient 
$\Delta$. Puisque $\dim \Delta'=\dim \Delta=\dim \Gamma+d$ 
il vient $\Delta'=\Delta$. 

Montrons maintenant que pour tout cellule $\Delta\in\mathscr D'_\Gamma$
et tout point~$\xi\in\mathring\Gamma$ suffisamment proche de~$x$, 
le point $(\xi,t)$ possède un antécédent $\eta$ sur $\Delta$.
Le point $y$ ayant été choisi dans $\mathring D_x$, 
le point $t=g_\trop(y)$ appartient à $g_\trop(\mathring D_x)$, qui est aussi 
l'intérieur topologique de $g_\trop(D_x)$ dans $\R^d$, car $D_x$ est de dimension $d$ ;
le point $t$ appartient
\textit{a fortiori} à l'intérieur topologique de $g_\trop(\Delta_x)$ dans $\R^d$. 
En appliquant le corollaire
\ref{coro.projection-voisinage} à la cellule $h_\trop(\Delta)$, à la projection 
de $\R^m\times \R^d$ sur $\R^d$ et au point $h_\trop(y)$, on en déduit que pour tout $\zeta$
suffisamment proche de $f_\trop(x))$ dans $f_\trop(\Gamma)$, le point 
$(\zeta,t)$ appartient à $h_\trop(\Delta)$. Par suite, pour tout point $\xi$ suffisamment proche de $x$ 
dans $\Gamma$, le point $(\xi,t)$ possède un antécédent sur $\Delta$ par $g'$.

Terminons maintenant la démonstration du lemme. 
Comme $\mathscr D_\Gamma$ est fini, il existe un voisinage~$W$ de~$x$ dans~$\Gamma$
tel que pour tout $\xi\in\mathring \Gamma\cap W$,
le point $(\xi,t)$ a un antécédent~$\eta$, et un seul, 
dans chaque cellule~$\Delta$ 
appartenant à~$\mathscr D'_\Gamma$; notons-le~$\eta_\Delta$.
Si de plus $\xi\in W\cap U$ alors $\eta_\Delta$ appartient à 
$\mathring \Delta$ ; en conséquence, $\deg^{\eta_\Delta}(g') = e_\Delta/d_\Gamma$
(lemme~\ref{lemm.dimensions-tropicales-fibre}). 
Ainsi, pour tout point $\xi\in\mathring\Gamma\cap W\cap U$, 
\[ \deg_{(\xi,t)}(g'|_V) 
= \sum_{\Delta\in\mathscr D'_\Gamma} \deg^{\eta_\Delta}(g'|_V)= 
\sum_{\Delta\in\mathscr D_\Gamma} \deg^{\eta_\Delta}(g')=\sum_{\Delta\in\mathscr D_\Gamma}
e_\Delta/d_\Gamma,\]
et cela conclut la démonstration
car $(\xi,\tau)\mapsto \deg_{(\xi,\tau)}(g'|_V)$ est un représentant du germe
$\deg^{(y)}(g')$ en $(x,t)$. 
\end{proof}

Reprenons la démonstration de la proposition. On a donc pour tout  $\xi$
suffisamment proche de $x$ dans $\mathring \Gamma$ les égalités
\[u(\xi)= \sum_{D\in \mathscr D'_C}v_D(x)\sum_{\substack{\Delta\in \mathscr D'_\Gamma\\D\subset \Delta}} (e_\Delta/d_\Gamma)
=\sum_{D\in \mathscr D'_C}v_D(x)\deg^{(y)}_{(\xi,t)}(g').\]
Comme l'ensemble des cellules de $\mathscr C$ qui contient $x$ est fini, il s'ensuit que
\[u(\xi)=\sum_{D\in \mathscr D'_C}v_D(x)\deg^{(y)}_{(\xi,t)}(g')\] pour tout $\xi$ 
dans $\Sigma_f$ suffisamment proche de $x$. 

Montrons b). On suppose que les fonctions de la classe $\mathscr B$ sont continues, que $p$ est séparé et que $\omega$ est faiblement positive.
Les fonctions $v_D$ sont alors continues et positives, et
$g'$ est séparé. Le germe $(\xi,\tau)\mapsto \deg^{(y)}_{(\xi,\tau)}(g')$ 
est semi-continu supérieurement (\ref{def-degre-bizarre}, c'est ici que sert
la séparation de $g'$). La fonction $u$ est donc semi-continue supérieurement
au voisinage de $x$. 
Celui-ci ayant été choisi arbitrairement sur $\Sigma_f$, la fonction $u$ est semi-continue supérieurement, ce qui montre \emph b). 

Montrons enfin \emph c). Supposons que $p$
est sans bord. Et supposons simplement pour le moment que  les fonctions de la classe $\mathscr B$ sont continues, 
et démontrons que $u|_{\Sigma_f}$ est continue. Par notre hypothèse sur $\mathscr B$ les fonctions $v_D$ sont  continues. 
Le morphisme $p$ est sans bord, et est composé de $g'$ et de la première projection ; par conséquent, $g'$ est sans bord. 
Il est donc fini en tout point de $(g')^{-1}(\Sigma_{f'})$ (lemme \ref{lemme.finiplat-local}, \emph b)).
Par ailleurs, pour tout $x\in X$, le morphisme 
$g|_{Y_x}$ est (fini et) plat en tout point situé au-dessus du squelette standard de $\gm[\hr x]^d$
 (lemme \ref{lemme.finiplat-local}, \emph c)) ; puisque $p$ est plat, 
 le critère de platitude par fibres (\cite{ducros2018}, théorème 8.3.8)
 assure que $g'$ est plat en tout point de $(g')^{-1}(\Sigma_{f'})$. 
 Le morphisme $g'$ étant fini et plat en tout point de $(g')^{-1}(\Sigma_{f'})$, 
 il résulte de \ref{def-degre-bizarre}
que le germe $(\xi,\tau)\mapsto \deg^{(y)}_{(\xi,\tau)}(g')$ 
est constant au voisinage de $(x,t)$. La fonction $u$ est donc continue
au voisinage de $x$. 
Celui-ci ayant été choisi arbitrairement sur $\Sigma_f$, la fonction $u$ est continue. 

Supposons maintenant que $\mathscr B=\mathscr C^r$ pour un certain $r$ appartenant à $[0,\infty]$. Pour toute
cellule $C\in \mathscr C$, les fonctions $u$ et $u_C$ coïncident sur $\mathring C$. Comme $u|_C$ et $u_C$ sont continues, 
$u$ et $u_C$ coïncident sur $C$. La fonction $u_C$ étant de classe $\mathscr C^r$, la fonction $u|_C$ l'est aussi. Ceci valant 
quel que soit $C$, la fonction $u$ est G-localement de classe $\mathscr C^r$. 
\end{proof}

En application du théorème précédent, nous pouvons donner
une variante relative du théorème selon les espaces sans bord
sont très riches.

\begin{theo}\label{theo-tresriche-relatif}
Soit $p\colon Y\to X$ un morphisme sans bord 
entre espaces $k$-analytiques presque bons et topologiquement séparés.
On suppose que $Y$ est purement de dimension~$n$
et localement holomorphiquement séparé, 
que $X$ est purement de dimension~$m$, 
et que les fibres de~$p$ sont purement de dimension~$d=n-m$. 
Supposons que l'espace~$X$, vu comme espace G-tropical de niveau~$m$,
est très riche. 
Alors, considéré comme espace G-tropical de niveau~$n$,
l'espace~$Y$ est très riche.
\end{theo}

\begin{proof}
Soit $\Sigma$ une partie squelettique compacte de~$V$. 
Il faut démontrer que $\Sigma$ est contenue dans une partie
squelettique compacte de~$V$ purement de dimension~$n$, 
ce qui peut se faire G-localement sur~$\Sigma$. 
D'après la proposition~9.18 de~\cite{ducros-t2024}, ou plutôt sa preuve, 
il existe un G-recouvrement fini~$(\Sigma_i)$ par des parties
squelettiques compactes et, pour tout~$i$,
des domaines compacts~$V_i$ de~$X$ et~$W_i$ de~$p^{-1}(V_i)$,
des moments $f_i\colon V_i\to \gm^{m}$ et $g_i\colon W_i\to \gm^d$,
tels que $\Sigma_i\subset \Sigma_{p^*f_i,g_i}$.
On peut ainsi supposer qu'il existe des domaines compacts~$V$ de~$X$ et~$W$ 
de~$p^{-1}(V)$,
des moments $f\colon V\to \gm^{m}$ et $g\colon W\to \gm^d$,
tels que $\Sigma\subset \Sigma_{p^*f,g}$.
Alors, $p(\Sigma)\subset\Sigma_f$ et pour tout $x\in\Sigma_f$,
$\Sigma\cap Y_x\subset \Sigma_{g|_{Y_x}}$
(lemme~\ref{lemm.dimensions-tropicales-fibre}). 

Soit $y\in\Sigma$ et soit $x=p(y)$.
Comme $X$ est presque bon en~$x$, il existe un voisinage ouvert~$V'$
de~$x$ dans~$X$ et un moment~$f'\colon V'\to\gm^m$
tel que $\abs{f_i-f'_i} <\abs{f_i}$ sur $V\cap V'$.
Comme $Y$ est presque bon en~$y$, il existe un voisinage ouvert~$W'$
de~$y$ dans~$p^{-1}(V')$ 
et un moment $g'\colon W'\to \gm^d$ tel que $\abs{g_j-g'_j}<\abs{g_j}$
sur $W\cap W'$.
Alors $\Sigma_{p^*f,g}$ et $\Sigma_{p^*f',g'}$ coïncident sur $W'$.
En remplaçant~$X$ par~$V'$ et~$f$ par~$f'$, $Y$ par~$W'$ et~$g$ par~$g'$,
$\Sigma$ par un voisinage compact de~$y$ dans~$\Sigma$, 
et en utilisant la compacité de~$\Sigma$,
on se ramène ainsi au cas où $f$ et~$g$ sont définis sur~$X$ et~$Y$.

La partie~$\Sigma_f$ de~$X$ est squelettique 
et l'espace~$X$ est supposé très riche.
Il existe donc une partie squelettique~$\Tau$ de~$X$ 
qui est purement de dimension~$n$ et qui contient
la partie squelettique compacte~$p(\Sigma)$ de~$\Sigma_f$.

En raisonnant à nouveau localement sur~$X$
et en appliquant la proposition~\ref{prop.parties-polyedrales-moments},
on se ramène au cas où il existe un moment~$f'\colon X\to\gm^m$ sur~$X$ 
tel que $\Tau\subset\Sigma_{f'}$.
Pour tout point $y\in\Sigma$, on
a $p(y)\in\Tau\subset \Sigma_{f'}$ et $y\in\Sigma_{f|_{Y_{p(y)}}}$,
donc $y\in\Sigma_{p^*f',g}$
(lemme~\ref{lemm.dimensions-tropicales-fibre}). 
Cela démontre que $\Sigma$ est contenu dans $\Sigma_{p^*f',g}$.
Quitte à remplacer~$f$ par~$f'$, on peut supposer
que $p(\Sigma)\subset \Tau\subset \Sigma_f$.

Notons $\alpha$ la $(m,m)$-forme sur~$X$ 
image réciproque par~$f_\trop$ de la forme strictement positive 
$\di x_1\wedge\dc x_1 \wedge \dots$ sur~$\R^m$.
De même, notons $\omega$ la $(d,d)$-forme sur~$Y$,
image réciproque par~$g_\trop$ de la forme strictement positive
analogue sur~$\R^d$.

\begin{lemm}
Soit $\lambda$ une fonction positive sur~$Y$,
G-tropicalement continue, à support compact,
telle que $\lambda|_{\Sigma_{p^*f,g}}\neq0$.
Alors $\int_Y \lambda\omega\wedge p^*\alpha >0$.
\end{lemm}
\begin{proof}[Démonstration du lemme]
Soit $y\in \Sigma_{p^*f,g}$ tel que $\lambda(y)\neq 0$;
posons $x=p(y)$.

Démontrons d'abord que que $\int_{Y_x} \lambda\omega >0$.
La forme $\omega|_{Y_x}$ est tropicalisée par~$g|_{Y_x}$;
le squelette $\Sigma_{g|_{Y_x}}$ en est un support fort;
il est purement de dimension~$d$ car $Y_x$ est sans bord
(corollaire~\ref{coro.Sigmaf=Sigmafn}).
Comme $\omega$ est l'image réciproque par~$g$ 
d'une forme strictement positive sur~$\R^d$, 
il en résulte que $\Sigma_{g|_{Y_x}}$ est égal au support 
de la mesure positive $(\omega|_{Y_x})_\Rad$ sur~$Y_x$.
Puisque le support de~$\lambda(y)>0$,
on en déduit que $\int_{Y_x}\lambda\omega >0$, comme annoncé.

Le morphisme~$p$ est sans bord;
d'après la proposition~\ref{prop.variation-cellule},
la fonction positive $u_{\lambda\omega}$
sur~$X$ définie par $u_{\lambda\omega}(t)=\int_{Y_t}\lambda\omega$ 
est continue sur~$\Sigma_{f}$.
Elle est donc strictement positive sur un voisinage de~$x$ dans~$\Sigma_{f}$.
Comme $\alpha$ est l'image réciproque par~$f_\trop$ 
d'une forme strictement positive sur~$\R^m$, 
le squelette $\Sigma_{f}$ est le support de la mesure positive~$\alpha_\Rad$.
Par conséquent, $\int_{\Sigma_{f}} u_{\lambda\omega}\alpha >0$,
c'est-à-dire $\int_X u_{\lambda\omega}\alpha>0$.
D'après le théorème~\ref{theo.integrale-fibre},
on en déduit
$\int_Y \lambda\omega\wedge p^*\alpha =\int_Y u_{\lambda\omega}\alpha>0$,
ce qu'il fallait démontrer.
\end{proof}

Terminons maintenant la démonstration du théorème.
Soit $W$ un voisinage compact de~$\Sigma$ dans~$Y$
et soit $\mathscr C$ une décomposition cellulaire 
en $(p^*f,g)$-cellules du squelette compact $\Sigma_{p^*f,g}\cap W$ 
qui est adaptée à~$\Sigma$.
Soit $C\in\mathscr C$ une cellule maximale qui contient une cellule
de~$\Sigma$; il suffit de prouver que $\dim(C)=n$.
Soit $y\in\mathring C$.

Par maximalité de~$C$,
la cellule ouverte~$\mathring C$ est ouverte dans~$\Sigma_{p^*f,g}$;
il existe donc un ouvert~$U$ de~$Y$ tel que $\mathring C=\Sigma_{p^*f,g}\cap U$.
Il existe alors une fonction lisse~$\lambda$ sur~$Y$,
à support compact contenu dans~$U$
telle que $\lambda(y)=1$
(corollaire~\ref{coro.berk-lisse-separe}).

D'après le lemme précédent, $\int_{Y}\lambda\omega\wedge p^*\alpha>0$,
si bien que $\lambda\omega\wedge p^*\alpha\neq0$.
Par construction, la cellule~$C$ est un support fort de
la $(n,n)$-forme $\lambda\omega p^*\alpha$.
Cela entraîne $\dim(C)=n$.
\end{proof}

\section{Continuité des intégrales fibre : cas général}

\begin{prop}\label{prop.polyedre-calibre-controle}
Soit $d$ un entier et soit $p\colon Y\to X$ un morphisme plat 
et compact entre espaces $k$-analytiques, 
à fibres purement de dimension~$d$. Soit $f\colon Y\to \gm^d$ un moment. 
Il existe un G-recouvrement $(X_i)$ de $X$ par des domaines analytiques
et, pour tout $i$, 
un moment $g_i$ sur $X_i$ tel que pour tout $x\in X_i$, le polyèdre de niveau $d$ calibré 
$f_\trop(Y_x)^{(d)}$ ne dépende que de $g_{i,\trop}(x)$.
\end{prop}

\begin{proof}
L'énoncé à montrer est tautologiquement G-local sur $X$, ce qui permet de supposer ce dernier affinoïde ; l'espace $Y$ est alors compact. 
En particulier, $X$ et~$Y$ sont topologiquement séparés;
les fibres de~$p$ sont considérées comme des espaces
G-tropicaux de niveau~$d$.

Soit $f'$ le $X$-morphisme $(p,f)$ de $Y$ vers $X\times_k \gm^d$. 
Soit $x$ un point de $X$. Le polyèdre $f_\trop(Y_x)^{(d)}$ est égal 
à $f_\trop(\Sigma_{f|_{Y_x}}^{(d)})$. Quant à son calibrage, il est égal au calibrage standard de $\R^d$ 
multiplié par la fonction qui envoie un point $v$ de $\R^d$ sur 
l'entier $\deg_{\sq_x(v)}(f|_{Y_x})=\deg_{\sq_x(v)}(f')$. 

Il suffit donc de montrer l'existence d'un G-recouvrement
$(X_i)$ de $X$ et, pour tout $x$, d'un moment $g_i$ sur~$X_i$ tel que pour tout $x\in X_i$, l'application 
de $\R^d$ dans $\N$ qui envoie $v$ sur $\deg_{\sq_x(v)}(f')$ ne dépende que de $g_{i,\trop}(x)$. 
Pour tout $x\in X$ notons $S_x$ le squelette standard du tore
$\gm[\hr x]^d$ ; par construction, $\Sigma_{f|_{Y_x}}=(f')^{-1}(S_x)$. Soit $S$
la réunion des $S_x$ pour $x$ parcourant $X$.  Le sous-ensemble $S$ de 
$X\times_k \gm^d$ est
l'image de la section de Shilov de l'application $(\mathrm{Id}, \trop)$ de $X\times_k \gm^d$ vers
$X\times \R^d$ ; comme la section de Shilov est continue (lemme \ref{lemm-continuite-shilov})
et que $X\times_k \gm^d$ est séparé, $S$ est une partie fermée de $X\times_k \gm^d$.

Soit $s$ un point de $S$, et soit $x$ son image sur~$X$. Puisque $s\in S_x$, le morphisme $f|_{Y_x}\colon Y_x \to \gm[\hr x]^d$
est plat et de dimension relative nulle en tout point de $f^{-1}(s)$. Le morphisme $f'$
est alors de dimension relative nulle en tout point de $(f')^{-1}(S)$ ; de plus le morphisme~$p$ est plat, et la projection de $X\times_k \gm^d$ 
sur~$X$ est plate ce qui entraîne en vertu du critère de platitude par fibre (\cite{ducros2018}, théorème 8.3.8) que 
$f'$ est plat en tout point de $(f')^{-1}(S)$. Le lieu $U$
des points de $Y$ en lesquels $f'$ et plat et de dimension relative nulle est un ouvert de
Zariski de $Y$, qui contient
le compact $(f')^{-1}(S)$ d'après ce qui précède. Il est alors loisible de remplacer $Y$ par un voisinage
analytique compact de $(f')^{-1}(S)$ dans $U$, ce qui permet de supposer que 
$f'$ est plat et de dimension relative nulle. 

Soit $(Y_j)$ un recouvrement fini de $Y$ par des domaines affinoïdes. Pour tout ensemble non vide
$J$ d'indices, posons $Y_J=\bigcap_{j\in J}Y_j$ et choisissons un recouvrement affinoïde fini
$(W_{J,\ell})_\ell$ de $Y_J$. Pour tout ensemble non vide 
$\Lambda$ d'indices, posons $W_{J,\Lambda}=\bigcap_{\ell \in \Lambda}W_{J,\ell}$.
Comme $Y_J$ est contenu dans un domaine affinoïde
de $Y$, les $W_{J,\Lambda}$ sont tous affinoïdes.
Par conséquent pour tout point $z$ de 
$ X\times_k Z$ on a
\[\deg_z(f')=\sum_{J\neq\emptyset}(-1)^{\abs J-1}\deg_z(f'|_{Y_J})
=\sum _{J,\Lambda \neq \emptyset}(-1)^{\abs J+\abs \Lambda }\deg_z(f'|_{W_{J,\Lambda}}).\]
Il suffit dès lors de savoir démontrer la proposition pour chacun des domaines affinoïdes
$W_{J,\Lambda}$
(elle fournira pour chacun de
ces domaines $W$ un jeu de cartes tropicales qui G-recouvre $X$ et contrôle
les polyèdres calibrés $f_\trop(W_x)^{(d)}$, et il suffira alors de prendre un raffinement commun de ces jeux pour contrôler
les polyèdres calibrés $f_\trop(Y_x)^{(d)}$). Par conséquent, on peut supposer que $Y$ est affinoïde. 

Soit $x\in X$ et soit $v\in \R^d$. 
Soient $y_1,\ldots, y_n$ les antécédents de $\sq_x(v)$ sous $f'$. Comme
$f'$ est plat et de dimension relative nulle, il résulte du théorème 8.4.6 de \cite{ducros2018} (et du fait qu'un espace 
de dimension nulle est automatiquement de Cohen-Macaulay) qu'il existe pour tout $i$ un voisinage affinoïde $Y_i$ de 
$y_i$ dans $Y$ tel que $f'|_{Y_i}$ admette une factorisation $Y_i\to T_i\to X\times_k \gm^d$ où $T_i$ est un espace affinoïde, où
$Y_i\to T_i$
est fini et plat de rang constant $e_i$ et où $T_i\to X\times_k \gm^d$ est quasi-étale. Quitte à les
restreindre, on peut supposer que les
$Y_i$ sont deux à deux disjoints. Il existe alors un voisinage affinoïde $X'_v$ de $x$ dans $X$ et un
voisinage paralinéaire compact $P_v$ de $\R^d$ tels que $(f')^{-1}(\sq_\xi(w))$ soit contenu dans 
$\coprod Y_i$ pour tout $\xi\in X'_v$ et tout $w\in P_v$. Il suffit alors de démontrer l'existence d'un nombre fini 
de cartes tropicales G-recouvrant $X'_v$ contrôlant les fonctions $w\mapsto \deg_{\sq_\xi(w)}(f')$ de $P_v$ dans $\N$ pour $\xi$
appartenant à $X'_v$.
Supposons en effet l'avoir prouvé. 
Le compact  $f_\trop(Y)$ est contenu dans $\bigcup_{v\in \mathscr V} P_v$ 
pour un certain ensemble fini $\mathscr V$. 
Soit $X'$ l'intersection des $X'_v$ pour $v\in \mathscr V$ ; 
c'est un voisinage affinoïde de $x$ dans $X$. 
En prenant un raffinement commun
des traces sur~$X'$ des jeux fournis sur les différents $X'_v$
pour $v\in \mathscr V$, 
on obtient un jeu fini de cartes tropicales G-recouvrant $X'$ et
contrôlant les fonctions $w\mapsto \deg_{\sq_\xi(w)}(f')$ pour $w\in \R^d$
et $\xi \in X'$ (pour le voir, on remarque que $\deg_{\sq_\xi(w)}(f')$ est nul dès que $w$ n'appartient
pas à $f_\trop(Y)$, et que si $w$ appartient à $f_\trop(Y)$ il appartient alors à $P_v$ pour un certain
$v\in \mathscr V$). 

On peut donc supposer, en remplaçant $X$ par $X'_v$ et $Y$ par l'image réciproque sour $(p,f_\trop)$ de $X'_v\times P_v$,
que $Y$ est réunion disjointe de domaines affinoïdes $Y_i$ sur chacun
desquels $f'$ admet une factorisation comme ci-dessus. En raisonnant G-localement sur $Y$ comme on l'a fait plus haut (c'est
même ici plus simple puisque les $Y_i$ sont deux à deux disjoints), on se ramène au cas où $f'\colon Y\to X\times_k\gm^d$ admet une factorisation 
$Y\to T\to X\times_k Z$ où $T$ est affinoïde, où $Y\to T$ est fini et plat de degré constant $e$, et où
$T\to X\times_k Z$ est quasi-étale. 

Le théorème 5.3 de \cite{ducros2024} assure que pour tout entier $m$, le lieu $V_m$ des points de
$ X\times_k \gm^d$ en lesquels les fibres géométriques du morphisme quasi-étale
$T\to  X\times_k Z$ sont de cardinal $\geq m$
est un domaine analytique compact de $X\times_k \gm^d$ (notons qu'il existe $M$ tel que 
$V_m$ soit vide dès que $m>M$). 
Sur le domaine analytique $V_m\setminus V_{m+1}$ de $ X\times_k \gm^d$ les fibres géométriques de
$T\to X\times_k \gm^d$ sont de cardinal exactement $m$, et le degré de $f'$
au-dessus d'un point de $V_m\setminus V_{m+1}$ est dès lors égal à $me$. 

La fonction $w\mapsto \deg_{\sq_\xi(w)}(f')$ ne dépend donc que des fonctions 
$w\mapsto \mathbf 1_{V_m}(\sq_\xi(w))$ pour $0\leq m\leq M$. On conclut alors à l'aide du lemme
\ref{lem-description-1V}. \end{proof}

\begin{lemm}\label{lem-description-1V}
Soit $X$ un espace $k$-affinoïde, soit $d$ un entier
et soit $V$ un domaine analytique compact de 
$X\times_k\gm^d$. 
Pour tout $x\in X$, soit $\Upsilon_x$ la fonction
\[\R^d\mapsto \{0,1\}, t\mapsto \mathbf 1_V(\sq_x(t)).\]
Il existe une famille finie
$(a_i)_{i\in I}$ de fonction analytiques
sur~$X$ et, pour tout $x \in X$, 
un voisinage affinoïde $U_x$ de $x$ dans $X$ et un sous-ensemble $I_x$ de
$I$ tel que $a_i|_{U_x}$ soit inversible pour tout $i\in I_x$ et tel que
pour tout $\xi\in U_x$, la fonction $\Upsilon_\xi$
ne dépende que de la famille $(\abs{a_i(\xi)})_{i\in I_x}$. 
\end{lemm}

\begin{proof}
Soit $T=(T_1,\ldots, T_d)$
la famille des fonctions coordonnées sur 
$\gm^d$. Comme $V$ est compact il existe un réel $R>1$ tel que 
$V$ soit contenu dans le domaine affinoïde $X\times_k Z$ de
$\gm^d$, où $Z$ désigne la polycouronne 
compacte défini par la conjonction d'inégalités 
$R^{-1}\leq \abs {T_i}\leq R$ pour $i$ variant de $1$ à $d$. 
Par le théorème de Gerritzen-Grauert, le domaine analytique compact
$V$ de $X\times_k Z$ est une union finie $\bigcup_n V_n$ de domaines rationnels
non vides. 
La fonction caractéristique $\mathbf 1_V$ est égale
à $\sup_n \mathbf 1_{V_n}$. l suffit
donc de montrer le lemme pour chacun des $V_n$, ce qui permet de supposer que
$V$ est un domaine rationnel non vide
de $X\times_k V$. 

Donnons-nous un système d'inégalités
\[\abs {f_1}\leq \mu_1 \abs {f_0}\;\text{et}\;\ldots\;\text{et}\;\abs {f_p}\leq \mu_p \abs {f_0}\]
décrivant le domaine rationnel $V$ où les $f_j$ sont des fonctions analytiques sur
$X\times_k Z$ sans zéro commun et où les $\mu_j$
sont des réels strictement positifs. 

Soit $J$ un ensemble d'entiers compris entre $0$ et 
$N$, et soit $D$ le lieu d'annulation simultanée des $f_j$ sur $X\times_k Z$. 
Nous allons montrer que $\mathbf 1_V$ s'exprime au voisinage de $D$
en fonction des $\abs{f_j}$ pour $j\notin J$. 
Soit $b$
le minimum de $\abs{f_0}$ sur le domaine $V$ ; c'est un réel strictement positif. 
Soit $W$ le voisinage affinoïde de $D$ dans $X\times_k Z$
défini par
la conjonction d'inégalités $\abs{f_j}\leq B_j$ pour $j\in J$ où 
$B_j=b/2$ si $j=0$ et $B_j=\mu_jb/2$ sinon. Le domaine 
affinoïde $V \cap W$ est alors : 
\begin{itemize}
\item vide si $0\in J$ ; 
\item décrit en tant que domaine de $W$ par les inégalités 
\[\abs{f_0}\geq b\;\text{et}\; \abs{f_j}\leq\mu_j\abs{f_0}\]
pour tous les $j>0$ tels que $j\notin J$ si $0\notin J$.
\end{itemize}
La fonction $\mathbf 1_V|_W$ s'exprime donc bien en fonction des
$\abs{f_j}$ pour $j\notin J$. 

Soit $x\in X$. Pour tout $j$, la fonction analytique~$f_j$ 
sur $X\times_k Z$ s'écrit comme une série 
$\sum_N a_{j,N}T^N$ en notation multi-indicielle, où
les $a_{j,N}$ sont des fonctions analytiques sur $X$, où $N$ parcourt
$\Z^d$ et où $\|a_j{,N}\|\cdot R^{|N|}$ 
tend vers zéro quand $|N|$ tend vers l'infini. 
Soit $E$ l'ensemble
des indices $j$ tels
que $f_j|_{Z_x}\neq 0$, c'est-à-dire encore tel qu'il existe
$N$ vérifiant $a_{j,N}(x)\neq 0$. 
D'après le paragraphe
précédent et par propreté, il existe un voisinage affinoïde $U_x$ de $x$
dans $X$ tel que sur $U_x\times_k Z$
la fonction $\mathbf 1_V$ 
s'exprime en fonction des $\abs{f_j}$ pour $j\in  E$. 
Soit $\Sigma$ l'image de $U_x\times [-\log R, \log R]^d$ par l'application 
$(z,t)\mapsto \sq_z(t)$, qui est continue 
(lemme \ref{lemm-continuite-shilov}) . C'est une partie compacte de $U_x\times_k Z$, dont
la fibre $\Sigma_\xi$ en tout point $\xi$ de $U_x$ est le squelette standard de la poly-couronne $Z_\xi$. 

Pour tout $j$ appartenant à $E$, la fonction 
$f_j|_{Z_x}$ est non nulle et ne s'annule donc pas
sur le squelette $\Sigma_x$ ; elle y est dès lors minorée par un réel 
strictement positif $\eps_j$. Par propreté topologique on peut
supposer, quitte à diminuer les $\eps_j$ et à restreindre $U_x$, qu'il existe 
un réel $\eps >0$ tel que $\abs{f_j}\geq \eps$ identiquement sur 
$\Sigma$ pour tout $j\in E$. 
Pour tout $j\in E$, choisissons un ensemble fini 
$\mathscr N_j$ d'indices tel que
$\abs {a_{j,N}}\cdot r^N<\eps$ identiquement sur $U_x$
pour tout $N\notin \mathscr N_j$ et tout
$r\in [R^{-1},R]^d$.
Si $N$ est un élément de $\mathscr N_j$ tel que 
$a_{j,N}(x)=0$ on peut restreindre $U_x$ de sorte que $\abs{a_{j,N}}\cdot r^N<\eps$ identiquement sur $U_x$
pour tout $r\in [R^{-1},R]^d$,
puis retirer $N$ de l'ensemble $\mathscr N_j$. 
On se ramène ainsi au cas où $a_{j,N}(x)\neq 0$
pour tout $j\in E$ et tout $N\in \mathscr N_j$ puis, en restreignant encore $U_x$, 
au cas où  $a_{j,N}$ est inversible sur $U_x$ pour tout 
$N\in \mathscr N_j$ et tout $j\in E$. 
On a 
\[\max_N \abs{a_{j,N}}\cdot r^N=
\max_{N\in \mathscr N_j}\abs{a_{j,N}}\cdot r^N\]
identiquement sur $U_x$ pour tout
$r\in [R^{-1},R]^d$. 
Par conséquent 
\[\abs{f_j(\sq_\xi(t))}=\max_{N\in \mathscr N_j}\abs{a_{j,N}(\xi)}\cdot \exp(t)^N\]
pour tout
$\xi\in U_x$ et tout $t\in [-\log R, \log R]^d$. 
Puisque $\mathbf 1_V(\sq_\xi(t))$ est nul dès que $t$ n'appartient pas 
à $ [-\log R, \log R]^d$, 
la fonction $\Upsilon_\xi$ peut être décrite, lorsque $\xi\in U_x$, uniquement en termes
des $\abs{a_{j,N}(\xi)}$ pour $j\in E$ et $N\in \mathscr N_j$. 
Les fonctions $a_{j,N}$ pour$j\in E$ et $N\in \mathscr N_j$ étant inversibles sur $U_x$, 
ceci achève la démonstration. 
\end{proof}

\begin{theo}\label{theo-variation-integrales}
Soit $d$ un entier et soit
$p \colon Y\to X$ un morphisme d'espaces 
$k$-analytiques, topologiquement séparé,
plat et à fibres purement de dimension~$d$. 
Soit~$\mathscr B$ une classe de coefficients 
qui est soit la classe des fonctions boréliennes localement bornées, 
soit la classe~$\mathscr C^r$ pour un certain~$r$ 
appartenant à $\N\cup\{\infty\}$. 
Soit~$\omega$ une $(d,d)$-forme G-tropicalement de classe~$\mathscr B$ sur~$Y$,
à support relativement compact sur~$X$. 
Soit~$u_\omega$ l'application $x\mapsto \int_{Y_x}\omega$ de~$X$ vers~$\R$. 

\begin{enumerate}
\item Il existe un G-recouvrement $(X_i) $ de $X$ par des domaines analytiques compacts et, pour tout $i$, 
un moment $f_i$ sur $X_i$, une décomposition cellulaire $\mathscr C_i$ de $f_{i,\trop}(X_i)$, et pour tout 
$C\in \mathscr C_i$ une application $u_C$ de classe $\mathscr B$ sur $C$, tels que pour tout indice
$i$, toute cellule $C\in \mathscr C_i$ et tout point $x\in f_{i,\trop}^{-1}(\mathring C)$, on ait $u_\omega(x) =u_C(f_{i,\trop}(x))$.

\item L'application $u_\omega$ est G-tropicalement borélienne.

\item Si 
$\mathscr B\in\{\mathscr C^r\}_{0\leq r\leq \infty}$, si $p$ est séparé et si $\omega$ 
est positive, l'application $u_\omega$ est semi-continue supérieurement.

\item Si $p$ est sans bord et si $\mathscr B\in \{\mathscr C^r\}_{0\leq r\leq \infty}$, 
l'application $u_\omega$ est G-tropicalement de classe $\mathscr B$. 
\end{enumerate}
\end{theo}

\begin{proof}
Commençons par une observation qui nous sera utile pour simplifier la situation à étudier. 
Supposons que $u_\omega$ s'écrive $u_1+u_2$, où $u_1$ et $u_2$ sont 
deux fonctions sur $X$ satisfaisant les conclusions de a). 
Alors $u_\omega$ les satisfait aussi. En effet, si
$(X_i, f_i, \mathscr C_i, u_C)_{i,C}$ atteste que $u_1$
satisfait les conclusions de a), et si $(Y_j, g_j, \mathscr D_j, v_D)_{j,D}$ en atteste pour $u_2$, alors 
$(X_i\cap Y_j, (f_i,g_j)|_{X_i\cap Y_j}, \mathscr E_{i,j}, w_E)_{i,j,E}$ en atteste pour $u_\omega$, 
où $\mathscr E_{i,j}$ est pour tout $(i,j)$ une décomposition cellulaire de
$(f_i,g_i)_\trop(X_i\cap Y_j)$ raffinant la trace de $\mathscr C_i\times \mathscr D_j$, et où
$w_E=u_{p_1(E)}\circ p_1+v_{p_2(E)}\circ p_2$.

Nous allons tout d'abord
démontrer les assertions (a), (b) et (c). 
Comme elles sont de nature  G-locale, on peut supposer $X$ affinoïde ; la forme
$\omega$ est alors à support compact. 
Nous allons maintenant expliquer comment nous ramener au cas où
il existe un un moment
$f$ sur $X$ et un moment
$g\colon Y\to \gm^d$
tels que 
$\omega$
soit G-localement
tropicalisée par $(f\circ p,g)$ -- nous dirons plus
brièvement que $\omega$ est \textit{bien présentée}
relativement au morphisme $p$, et que le couple $(f,g)$ en atteste. 
Observons que si l'on exhibe un moment $g\colon Y\to \gm^\delta$
avec $\delta\leq d$ et un moment $f$ sur $X$ tel que $\omega$ soit 
G-localement tropicalisée par $(f\circ p,g)$ alors $\omega$ est bien présentée 
relativement à $p$ : le couple $(f',g)$ en atteste, où $f'$ est le moment obtenu en adjoignant
$d-\delta$ fonctions constantes égales à~$1$ au moment~$f$.

Soit $x$ un point de $X$. 
Soit $y$ un antécédent de $x$ sur $Y$ appartenant au support de $\omega$. 
Puisque $\omega$ est G-tropicalement de classe $\mathscr B$, il existe une famille
finie de domaines affinoïdes de $Y$ contenant $y$, dont la réunion est un voisinage de $y$, 
et sur chacun desquels $\omega$ est tropicale. 
Il s'ensuit, grâce à la proposition
 \ref{prop.projection-tropicale} et par
 compacité de $p^{-1}(x)\cap \supp(\omega)$, qu'il existe une famille finie 
 $(W_j)$ de domaines
 affinoïdes de $Y$ dont la réunion est un voisinage de
  $p^{-1}(x)\cap \supp(\omega)$ et, pour tout $j$, un voisinage 
 un voisinage affinoïde
$V_j$
de $x$ contenant $p(W_j)$
tel que $\omega|_{W_j}$ soit bien présentée
relativement  à $p|_{W_j}\colon W_j\to V_j$. Par propreté topologique 
de $p$ il existe 
un voisinage affinoïde $V$ de $x$ dans $X$, contenu dans chacun des $V_j$, et tel que $p^{-1}(V)$
soit contenu dans $\bigcup_j W_j$. En remplaçant chacun des $W_j$ par son intersection avec $p^{-1}(V)$, 
on se ramène au cas où les $V_j$ sont tous égaux à un même voisinage affinoïde $V$ de $x$. Pour tout ensemble
non vide $J$ d'indices, désignons par $u_J$ l'application 
de $V$ dans $\R$ qui envoie $v$ sur $\int_{\bigcap_{j\in J }W_{j,v}}\omega$. On a alors 
$u_\omega|_V=\sum_J(-1)^{|J|}u_J$. À l'aide de l'observation faite en début de preuve, et du caractère G-local sur $X$ des
différents énoncés, on voit qu'il suffit de montrer (a) et (b) pour chacun des morphismes $W_j\to V$, ce qui nous ramène
bien pour ces deux assertions au cas où $\omega$ est bien présentée.
Nous allons maintenant vérifier qu'on peut également se ramener au cas bien présenté pour démontrer (c). 

Pour tout ensemble non vide d'indices $J$, posons 
$W_J=\bigcap_{j\in J}W_j$ et $\Omega_J=W_J\setminus \bigcup_{j\notin J}
W_j$ ; les $\Omega_J$ non vides forment une partition de $\bigcup_j W_j$

Pour tout $J$, choisissons une famille $(K_{J a})_a$ de domaines analytiques compacts de $W_J$ contenant $\Omega_J$,
dont l'intersection est égale à $\overline{\Omega_J}$. Remarquons que pour tout $J$ et tout $v\in V$ 
on a $\int_{W_{J,v}}\mathbf 1_{\Omega_J}\omega=
\int_{W_{J,v}}\mathbf 1_{\overline{\Omega_J}}\omega$. En effet, considérons un support fort 
$\Sigma$ de $\omega_{W_{J,v}}$. L'intersection $(\overline{\Omega_J}\setminus \Omega_J)\cap \Sigma$
est alors contenus dans $\Sigma \cap \bigcup_{j\notin J}\partial (W_J\cap W_j)$, qui est lui-même inclus dans une partie
paralinéaire compacte de dimension $<d$ de $\Sigma$, ce qui permet de conclure. 

La fonction $u_\omega|_V$ est égale à $\sum_J \mathbf 1_{\Omega_J} u_\omega$, et donc
à  $\sum_J\mathbf 1_{\overline{\Omega_J}} u_\omega$ par ce qui précède.
Lorsque $\omega$ est positive, on peut par conséquent écrire 
$u_\omega$ comme l'infimum des fonctions de la forme 
$\sum_J \mathbf 1_{K_{J, a_J}}u_\omega$, sur toutes les familles $(a_J)_J$ possibles. Il suffit dès lors
pour montrer (c) de
s'assurer que $\mathbf 1_{K_{J,a}}u_\omega$ est semi-continue supérieurement pour tout $J$ et tout $a$, ce qui revient
à supposer que $Y$ est l'un des $K_{J, a}$ ; or tout $K_{J,a}$ est inclus dans $W_J$ et dès lors
dans l'un des $V_j$ (prendre $j$ dans l'ensemble non vide $J$)
si bien qu'on s'est également ramené pour l'assertion (c) au cas où $\omega$ est bien présentée relativement à $p$.  

On suppose donc désormais $\omega$ bien présentée relativement à $p$
et l'on choisit un couple $(f,g)$ en attestant. En vertu de la proposition \ref{prop.polyedre-calibre-controle}, on peut (en raisonnant G-localement sur $X$) 
supposer qu'il existe une famille finie de fonctions inversibles 
$h_1,\ldots, h_r$ sur $X$ telles que pour tout $x\in X$, le polyèdre de niveau $d$ calibré $g_\trop(Y_x)$ 
ne dépende que de $h_\trop(x)$. Mais alors $\int_{Y_x}\omega$ ne dépend que
de $(f_\trop, h_\trop)(x)$. 

Posons $P=(f_\trop, h_\trop)(X)$. 
Soit $Z$ une composante irréductible de $X$, vue comme un espace analytique réduit 
de niveau $\dim Z$. Soit $P_Z$ le polyèdre $(f_\trop, h_\trop)(Z)$.
Puisque $Z$ est très riche, il existe une partie 
squelettique compacte $\Sigma_Z$ de $Z$ telle que $(f_\trop, h_\trop)(\Sigma_Z)=P_Z$. Soit $\Sigma$ la réunion des
$\Sigma_Z$ pour $Z$ parcourant l'ensemble des composantes irréductibles de $X$ ; notons que cette réunion est disjointe, puisque
$\Sigma_Z$ est constitué de points Zariski-denses dans $Z$. On a alors $(f_\trop,h_\trop)(\Sigma)=P$. 
Les assertions à montrer sont G-locales sur $X$. Il suffit en particulier de les prouver sur les éléments
d'un G-recouvrement de $X$ image réciproque d'un G-recouvrement de $P$ ; on peut par conséquent supposer que
l'application paralinéaire surjective 
$(f_\trop, h_\trop)|_\Sigma\colon \Sigma \to P$ admet une section paralinéaire à valeurs dans $\Sigma_Z$ pour une certaine
composante irréductible $Z$ de $X$ ; en remplaçant $\Sigma$ par l'image de cette
section, on peut supposer que $\Sigma=\Sigma_Z$ et que 
$(f_\trop,h_\trop)$ induit un isomorphisme de $\Sigma$ sur $P$. 

Puisque $\int_{Y_x}\omega$ ne dépend que
de $(f_\trop, h_\trop)(x)$,
il existe une unique application $v$ de $P$ dans $\R$ telle que
$u_\omega=v\circ (f_\trop, h_\trop)$. Pour démontrer (a), (b) et (c), il suffit de prouver qu'il existe une 
décomposition cellulaire $\mathscr C$ de $P$, et pour tout $C\in\mathscr C$, une fonction $u_C$
de classe $\mathscr B$ sur $C$ telle que $v$ et $u_C$ coïncident sur $\mathring C$, et que $v$ est semi-continue
supérieurement si $\omega$ est positive. 
Comme $(f_\trop, h_\trop)$ induit un isomorphisme de $\Sigma$ sur $P$, il suffit de démontrer les propriétés analogues pour la fonction 
$u_\omega|_\Sigma$. Mais celles-ci se déduisent de la proposition \ref{prop.variation-cellule}
appliquée au morphisme de $Y\times_X Z$ vers $Z$ induit par $p$ (notons que $Y\times_X Z$ et $Z$ sont équidimensionnels). 

Supposons maintenant que $p$ est sans bord et que $\mathscr B\in \{\mathscr C^r\}_{0\leq r\leq
\infty}$, et montrons que $u_\omega$ est G-tropicalement 
de classe $\mathscr B$.
Nous allons devoir faire un peu de \textit{bootstrapping}. L'assertion 
(a) déjà démontrée, couplée à un raisonnement G-local sur $X$, permet de supposer que $X$ est 
compact et qu'il existe 
un moment $h\colon X\to \gm^n$ tel que $u_\omega=v\circ h_\trop$ pour une certaine fonction $v$ de $h_\trop(X)$ vers $\R$. Posons $P=h_\trop(X)$. 
Soit $Z$ une composante irréductible de $X$, vue comme un espace analytique réduit 
de niveau $\dim Z$. Soit $P_Z$ le polyèdre $h_\trop(Z)$.
Puisque $Z$ est très riche, il existe une partie 
squelettique compacte $\Sigma_Z$ de $Z$ telle que $h_\trop(\Sigma_Z)=P_Z$. Soit $\Sigma$ la réunion des
$\Sigma_Z$ pour $Z$ parcourant l'ensemble des composantes irréductibles de $X$ ; notons que cette réunion est disjointe, puisque
$\Sigma_Z$ est constitué de points Zariski-denses dans $Z$. On a alors $h_\trop(\Sigma)=P$. 
L'assertion (c) à montrer  est G-locale sur $X$, et il suffit en particulier de la prouver sur les éléments
d'un G-recouvrement de $X$ image réciproque d'un G-recouvrement de $P$ ; on peut par conséquent supposer que
l'application paralinéaire surjective 
$h_\trop|_\Sigma\colon \Sigma \to P$ admet une section paralinéaire à valeurs dans $\Sigma_Z$ pour une certaine
composante irréductible $Z$ de $X$ ; en remplaçant $\Sigma$ par l'image de cette
section, on peut supposer que $\Sigma=\Sigma_Z$ et que 
$h_\trop $ induit un isomorphisme de $\Sigma$ sur $P$. 

Pour démontrer (c), il suffit de prouver que $v|_P$ est 
G-localement de classe $\mathscr B$. 
Comme $h_\trop$ induit un isomorphisme de $\Sigma$ sur $P$, il suffit de démontrer 
que
$u_\omega|_\Sigma$ est G-localement de classe $\mathscr B$. Pour ce faire on peut remplacer
$X$ par un voisinage analytique compact de $\Sigma$ ne rencontrant aucune autre composante 
irréductible de $X$, ce qui permet de supposer ce dernier équidimensionnel. Enfin, en raisonnant une fois encore
G-localement sur $X$, on peut le supposer affinoïde. 

Comme $X$ est affinoïde et comme
$p$ est sans bord, l'espace $Y$ est bon. Soit $x$ un point de $X$ et soit $y$ un point de
$p^{-1}(x)\cap \supp(\omega)$. Comme $Y$ est bon il est décent en $y$ et il résulte donc
de la proposition \ref{prop.local-G-local} qu'il existe un voisinage affinoïde $W_y$ de $y$, un moment $h_y$ sur
$W_y$ et une G-forme $\alpha_y$ de bidegré $(d,d)$ et de classe $\mathscr B$ sur $h_{y,\trop}(W_y)$ telle que
$\omega|_{W_y}=h_y^*\alpha_y$. En vertu de la proposition  \ref{prop.projection-tropicale} on peut alors supposer, 
quitte à retreindre
$W_y$, qu'il existe un voisinage affinoïde $V_y$ de $x$ contenant $p(W_y)$
tels que $\omega|_{W_y}$ soit bien présentée relativement 
à $p|_{W_y}\colon W_y\to V_y$. 
Par compacité, il existe $y_1,\ldots, y_m$ tels que les
intérieurs des $W_{y_i}$ recouvrent $p^{-1}(x)\cap \supp(\omega)$. Choisissons un voisinage affinoïde
$V$ de $x$ contenu dans tous les $V_{y_i}$ tel que $p^{-1}(V)\cap
\supp(\omega)$ soit inclus dans la
réunion des intérieurs topologiques des $W_{y_i}$,
ce qui 
est possible par propreté topologique
de $p$.
En remplaçant chacun des $W_{y_i}$
par $W_{y_i}\times_{V_{y_i}}V$ on se ramène alors au cas où les $V_{y_i}$ sont tous égaux à un même voisinage affinoïde
$V$ de $x$, et où $p^{-1}(V)\cap \supp(\omega)$
est contenu
dans $\bigcup \mathring{W_{y_i}}$.  Puisque $Y$ est bon, il existe une partition de l'unité
$(\lambda_i)$ au voisinage de $p^{-1}(x)\cap \supp(\omega)$ 
subordonnée au recouvrement de $p^{-1}(x)$ par les 
$\mathring{W_{y_i}}$. Soit $u_i$ l'application de $V$ dans $\R$ qui envoie $v$
sur $\int_{(\mathring W_{y_i})_v}\lambda_i \omega$. On a alors $u_\omega|_V=\sum u_i$.

Il suffit donc de démontrer que pour tout $i$ la fonction 
$u_i|_\Sigma$ est G-localement de classe 
$\mathscr B$. On peut donc supposer que $\omega$ est de la forme $\lambda\eta$ où $\lambda$ est lisse
et à support compact et où $\eta$ est une $(d,d)$-forme G-tropicalement de classe $\mathscr B$
bien présentée relativement à $p$. Choisissons un couple $(f,g)$ qui atteste de la bonne présentation 
de $\eta$ ; quitte à raisonner G-localement sur $X$ et à agrandir le moment $f$ et le squelette $\Sigma$ on peut 
supposer que $\Sigma=\Sigma_f$; le résultat voulu découle alors de la proposition \ref{prop.variation-cellule}.
\end{proof}

\begin{coro}
Soit $d$ un entier et soit $p\colon Y\to X$ un morphisme 
d'espaces $k$-analytiques topologiquement séparés qui est
plat et purement de dimension relative~$d$. 
Soit~$\omega$ une $(d,d)$-forme à coefficients
tropicalement boréliens localement bornés sur~$Y$, 
et à support relativement compact sur~$X$. 
Soit~$u$ la fonction numérique $x\mapsto \int_{Y_x}\omega$ sur~$X$. 

\begin{enumerate}
\item Soit $Z$ un fermé de Zariski de $X$ tel que $u$ soit constante sur $Z$. 
La fonction $u$ est alors constante au voisinage de $Z$. 

\item La fonction $u$ est constante au voisinage de tout point rigide de $X$. 
\end{enumerate}
\end{coro}

\begin{proof}
Commençons par montrer (a). L'assertion est G-locale sur $X$. 
On peut donc d'après le théorème \ref{theo-variation-integrales} (a) supposer que $X$
est compact et qu'il existe un moment $f$ sur $X$ et une application $v$ de $f_\trop(X)$ vers $\R$
telle que $u=v\circ f_\trop$. Posons $Q=f_\trop(Z)$. C'est un polyèdre compact contenu dans $P$, et
$v$ est constante sur $Q$. L'image réciproque $f_\trop^{-1}(Q)$
est alors un domaine analytique de
$X$ contenant $Z$ sur lequel $u$ est constante. 
Comme $Z\hookrightarrow X$ est sans bord, $Z$ est contenu 
dans $\Int(f_\trop^{-1}(Q)/X)$ qui est aussi l'intérieur topologique de $f_\trop^{-1}(Q)$ dans $X$. 
Autrement dit,  $f_\trop^{-1}(Q)$ est un voisinage de $Z$ dans $X$. 

Si $x$ est un point rigide de $X$, le sinlgeton $\{x\}$ est un
fermé de Zariski de $X$. L'assertion (a) assure alors que $u$
est constante au voisinage de $x$. 
\end{proof}

\begin{rema}
L'hypothèse que $p$ est plat est indispensable
dans le théorème~\ref{theo-variation-integrales} et son corollaire.
En effet, dans le cas du morphisme de normalisation d'une singularité double
ordinaire en dimension~$1$ et à la fonction constante~$\omega=1$, 
l'intégrale dans la fibre vaut~$2$ au point singulier et~$1$ ailleurs.

Nous considérons cet énoncé comme l'analogue non archimédien
du théorème de~\cite{Stoll-1966a} en géométrie analytique complexe. 
Il y a cependant plusieurs différences formelles :
\begin{itemize}
\item Nous faisons une hypothèse de platitude sur le morphisme~$p$
et intégrons exactement dans la fibre, avec sa structure analytique
adaptée ;
\item L'énoncé en géométrie complexe ne suppose pas de platitude sur~$p$, 
mais la base~$X$ est supposée normale ; de plus,
les fibres de~$p$ sont considérées comme des espaces analytiques réduits,
si bien que les intégrales doivent être pondérées par une multiplicité adéquate.
\end{itemize}
\end{rema}

 \section{Les espaces de Berkovich stricts sont très riches}
 
 \begin{prop}\label{prop-extension-squelette}
 Soit~$X$ un espace  $k$-analytique de niveau~$m$ et soit~$x$ un point de~$X$
 tel que~$\hrt x^1$ soit de degré de transcendance~$m$ sur~$\widetilde k^1$. 
 Soit~$Y$ un espace~$k$-analytique de niveau~$n$ et soit~$f$ un morphisme
 de~$Y$ vers~$X$. 
 
 \begin{enumerate}
\item La fibre~$f^{-1}(x)$ est de dimension~$\leq n-m$ ; considérons-la comme un espace~$k$-analytique de niveau~$n-m$. 

\item Soit~$d$ un entier et soit~$\Sigma$ une partie squelettique compacte de~$f^{-1}(x)$ purement de dimension~~$d$. Il existe une partie squelettique compacte~$\Tau$ de~$Y$ purement de dimension~$d+m$ et 
contenant~ $\Sigma$.
 \end{enumerate}
 \end{prop}
 \begin{proof}
 Puisque~$Y$ est de niveau~$n$ il est de dimension~$\leq n$. 
 Si~$t\in f^{-1}(x)$ on a donc~$d_k(t)\leq  n$, soit encore
~$d_{\hr x}(t)+m\leq n$, et donc~$d_{\hr x}(t)\leq n-m$. Ainsi 
~$\dim f^{-1}(x)\leq n-m$, d'où \emph a). 
 
 Montrons maintenant \emph b). La proposition \ref{prop-squelettes-nilpotents}
assure que le problème est insensible à la présence de nilpotents. On peut donc 
remplacer~$Y$ et~$X$ par les espaces réduits associés. 
Il suffit de montrer l'assertion en remplaçant~$X$ par n'importe quel domaine
analytique de~$X$ contenant~$x$. Cela permet de supposer~$X$ affinoïde. 

Faisons une remarque générale.
Soit~$(Y_i)$ une famille finie de domaines analytiques
compacts de~$Y_x$.
dont la réunion contient~$\Sigma$. La partie squelettique~$\Sigma$ est alors
recouverte par les~$(Y_i\cap\Sigma)^{(d)}$, et il suffit de montrer que chacun
des~$(Y_i\cap\Sigma)^{(d)}$ est contenu dans une partie squelettique compacte 
de~$Y$ purement de dimension~$n$ ; il suffit donc de montrer l'énoncé voulu pour chacun 
des~$(Y_i\cap \Sigma)^{(d)}$. 

En recouvrant~$\Sigma$ par une famille finie de domaines affinoïdes de~$Y$,
cette remarque permet d'abord de supposer que $Y$ est affinoïde.
Elle permet également de supposer qu'il existe un domaine analytique compact 
$W$ de~$Y_x$ contenant~$\Sigma$
et un moment~$g\colon W\to \gm[\hr x]^{n-m}$ tel que 
$g_\trop|_{\Sigma}$ soit injectif. 
La partie paralinéaire compacte~$\Sigma$ est alors
contenue dans~$\Sigma_g$ ; elle est donc de la forme~$W'\cap \Sigma_g$
pour un certain domaine analytique compact~$W'$ de~$W$ 
(lemme \ref{lemme-proprietes-paralineairesX} \emph e))  ; quitte à remplacer
$W$ par~$W'$ on peut supposer que~$\Sigma=\Sigma_{g}$. 

Comme~$Y$ est affinoïde, 
il existe pour tout point~$y$ de~$\Sigma$ un voisinage affinoïde 
$V_y$ de~$y$ dans~$Y$ et un moment~$g_y~$ sur~$V_y$ tel que~$\abs{g_y-g}<\abs g$ sur 
$V_y\cap W$ ; 
on a alors~$g_\trop=g_{y,\trop}$ sur~$V_y\cap W$, 
et~$\Sigma_g\cap(V_y\cap W)=\Sigma_{g_y}\cap (V_y\cap W)$. 
Par compacité de~$\Sigma$, ce dernier est contenue dans une union finie 
de~$V_y$. Ceci permet au vu de la remarque générale ci-dessus de supposer que 
$g$ s'étend en un moment défini sur~$Y$ tout entier et encore noté~$g$. 

Le domaine analytique compact~$W$ de~$Y_x$ est une union finie de domaines rationnels de
$Y_x$. En appliquant encore la remarque générale, on 
peut ainsi supposer que~$W$ est un domaine rationnel de~$Y_x$. 
Choisissons un système rationnel d'inégalités~$\abs{f_i}\leq \lambda_i \abs h$ décrivant
$W$. Les~$f_i$ et~$h$ étant sans zéro commun sur~$Y_x$, la fonction~$h$ ne s'annule pas
sur~$W$, si bien qu'on peut qu'on peut rajouter
une inégalité~$\abs h \geq r$ au système ci-dessus. 

Mais alors si les~$f'_i$ et~$h'$ sont des fonctions
analytiques sur~$Y_x$ telles que~$\abs{f'_i-f_i}<\lambda_i r$
et~$\abs {h'}<r$ en tout point de~$Y_x$, 
le système d'inégalités~$\{\abs{f'_i}\leq \lambda_i\abs {h'}\}_i, \abs{h'}\geq r$
définit encore le domaine~$W$. 
Quitte à restreindre~$X$, on peut donc supposer 
que $h$ et les~$f_i$ 
sont toutes des restrictions de fonctions analytiques sur~$Y$ ; 
le domaine~$W$ s'écrit alors~$V\cap Y_x$ pour un certain domaine affinoïde (rationnel)
$V$ de~$Y$. 
En remplaçant~$Y$ par~$V$, on se ramène ainsi au cas où~$\Sigma=\Sigma_{g|_{Y_x}}$. 

Il résulte du théorème 10.3.7 de \cite{ducros2018} et du fait que 
$X$ est réduit que
le morphisme~$f$ est plat en tout point de~$Y_x$. 
Comme~$\Sigma=\Sigma_{g|_{Y_x}}$, le morphisme
$g|_{Y_x} \colon Y_x \to \gm[\hr x]^{n-m}$ est plat en tout point de~$\Sigma$ ; 
en vertu du critère de platitude par fibres
(\cite{ducros2018}, théorème 8.3.8) 
le~$X$-morphisme~$g\colon Y\to X\times \gm^{n-m}$ est plat en tout point de~$\Sigma$. 
Quitte à remplacer~$Y$ par un voisinage analytique compact de~$\Sigma$ (l'espace~$Y$
peut perdre son caractère affinoïde à cette occasion, mais nous n'en aurons plus besoin)
on peut supposer que~$g\colon Y\to X\times \gm^{n-m}$ est plat. L'image~$g(Y)$
est alors un domaine analytique compact de~$X\times \gm^{n-m}$ (c'est un théorème dû
à Raynaud dans le cas strict, voir \cite{ducros2018}, théorème 9.2.1 pour le  cas général). 

Pour tout~$u\in X$, notons 
$S_u$ le squelette standard de~$\gm[\hr u]^{n-m}$. 
D'après le lemme \ref{lem-description-1V}, il existe un voisinage affinoïde~$U$ de
$x$ dans~$X$ et une famille finie~$(a_1,\ldots,a_\ell)$ de fonctions inversibles sur~$U$
telles que pour tout~$u\in U$ le polyèdre compact 
$\trop(g(Y)\cap S_u)$ ne dépende que des~$\abs{a_i(u)}$.

Puisque~$\hrt x^1$ est de degré de transcendance~$m$ sur~$\widetilde k^1$
et $d_k(x)\leq m$, il s'ensuite que $d_k(x)=m=\dim_x(X)$ et que
le groupe~$\abs{\hr x}^\times$ est de torsion modulo 
$\abs{k^\times}$, et le germe~$(X,x)$ est dès lors strict
(exemple \ref{exemple-torsion-strict}). 
Quitte à remplacer chacune des~$a_i$ par~$a_i^{n_i}/\lambda_i$
pour un entier~$n_i$ et un élément~$\lambda_i$ de~$k^\times$ convenables, 
on peut supposer que~$\abs{a_i(x)}=1$ pour tout~$i$. 

Soit~$M$ un modèle de~$\widetilde{(X,x)}^1$ ; 
c'est une~$\widetilde k^1$-variété intègre
de corps des fonctions~$\hrt x^1$.
Soit~$M'$ un ouvert de Zariski non vide de~$M$
sur lequel les~$\widetilde{a_i(x)}$ sont définies et inversibles. 
Soit~$X'$ un domaine analytique de~$X$ contenant~$x$ tel que
$\widetilde{(X',x)}^1=\zr{\hrt x^1}{\widetilde k^1}\{M'\}$. 
Soit~$\xi$ un point fermé de~$M'$ et soit~$(\beta_1,\ldots, \beta_m)$ 
une suite régulière de paramètres de~$\mathscr O_{M',\xi}$. On peut supposer, quitte à restreindre~$M'$, que les~$\beta_i$ sont définis sur~$M'$
et, quitte à restreindre~$U$,  qu'il existe pour tout~$i$ une fonction analytique inversible~$b_i$ sur~$U$
telle que~$\abs{b_i(x)}=1$ et~$\widetilde{b_i(x)}=1$ ; on peut enfin remplacer~$X'$
par~$X'\cap U$, c'est-à-dire faire en sorte que~$X'\subset U$. 

\begin{lemm}
Il existe une~$\widetilde k^1$-valuation sur~$\hrt x^1$ centrée en~$\xi$ et tel que~$v(\beta_i)$ soit pour tout 
$i$ infiniment petit relativement à~$\langle v(\beta_j)\rangle_{j<i}$. 
\end{lemm}
\begin{proof}
On procède par récurrence
sur~$m$ en supposant donc l'énoncé vrai en dimension~$<m$. Il n'y rien à faire si~$m=0$. Supposons~$m>0$. On
choisit une composante irréductible~$N$ du lieu des zéros de~$\beta_1$ qui contient~$\xi$, 
puis une composante irréductible~$N'$ de l'image réciproque de~$N$ sur la normalisée de~$M'$. 
La composante~$N'$ définit une valuation~$w$ sur~$\hrt x^1$ centrée en le point générique de~$N'$ ;
l'hypothèse de récurrence appliquée à la variété~$N$
qui est de dimension~$m-1$ et a un antécédent~$\xi'$ de~$\xi$ sur~$N'$ fournit une valuation~$w'$ sur le 
corps des fonctions de~$N'$, centrée en~$\xi'$, et la composée~$v$ de~$w$ et~$w'$ convient alors. 
\end{proof}

La remarque \ref{rem-description-trop-germ}
décrit le germe de cône~$b_\trop(X',x)$ en termes d'un sous-foncteur~$D$ de~$\mathsf D^m_{\{1\}}$
qui par construction contient~$(v(\beta_1),\ldots, v(\beta_m))$. Comme les~$v(\beta_i)$ sont~$\Z$-linéairement
indépendants,~$D$ est de dimension~$m$, si bien que la dimension tropicale de~$g|_{X'}$ en~$x$ est égale à~$m$, si bien que $x\in \Sigma_{g|_{X'}}{(m)}$.
Il existe en conséquence une~$g$-cellule~$C\subset \Sigma_g$ de dimension~$m$ contenant~$x$ et contenue dans~$X'$. 
Comme les~$\widetilde{a_i(x)}$ sont inversibles sur~$M'$, chacune des~$\abs{a_i}$ est identiquement égale à~$1$ sur~$X'$
et en particulier sur~$C$. 

Pour tout~$u$ appartenant à~$C$ on a~$\abs{a_i(u)}=1=\abs{a_i(x)}$ 
quel que soit~$i$, si bien que~$\trop(g(Y)\cap S_u)=\trop(g(Y)\cap S_x)
= g_\trop(\Sigma)$, où
la seconde égalité résulte du fait que~ $\Sigma=\Sigma_{g|_{Y_x}}$ ; notons que~$g_\trop(\Sigma)$ est une partie paralinéaire
compacte de~$\R^{n-m}$ purement de dimension~$d$. 

Il résulte du lemme \ref{lemm.dimensions-tropicales-fibre}
que~$\Sigma_{g,b}$ est la réunion des~$\Sigma_{g|_{Y_u}}$ pour~$u$
parcourant~$\Sigma_b$. 
Par conséquent la réunion~$\Sigma'$ des ~$\Sigma_{g|_{Y_u}}$
pour~$u$ parcourant~$C$ est la partie linéaire compacte de 
$\Sigma_{g,f^*b}$ égale à~$f^{-1}(C)$. 
Pour tout~$u\in C$ on a donc
a~$g_\trop(\Sigma_{g|_{Y_u}})=g_\trop(\Sigma)$, si bien que 
$(gf^*,b)_\trop(\Sigma')=g_\trop(\Sigma)\times b_\trop(C)$ ; 
c'est une partie paralinéaire purement de dimension~$d+m$ de~$\R^n$. 
La surjection paralinéaire~$(g,f^*b)_\trop$ de~$\Sigma'$
sur~$g_\trop(\Sigma)\times b_\trop(C)$
admet G-localement des sections
(proposition \ref{prop.sections-G-locales}). 
Choisissons un 
recouvrement fini~$(Q_i)$ de~$g_\trop(\Sigma)\times b_\trop(C)$ par des parties
paralinéaires compactes, chacune étant munie d'une section~$\sigma_i$ de ~$(g,
f^*b)_\trop$. 
Quitte à remplacer chacune des~$Q_i$ par~$Q_i^{(d+m)}$, 
on peut supposer chacune des $Q_i$ purement de dimension~$d+m$. 
Mais alors la réunion~$\Tau$ des~$\sigma_i(Q_i)$ est une partie
paralinéaire de~$\Sigma'$ purement de dimension~$d+m$.
Il reste pour conclure à s'assurer que~$\Tau$ contient~$\Sigma$. 
Soit donc~$y\in \Sigma$ et soit~$z$ son image par~$(g,f^*b)_\trop$; 
on a $z=(g_\trop(y), b_\trop(x))$. Soit~$i$ tel que~$z$ appartienne à~$Q_i$. 
On a alors~$(f^*b)_\trop(\sigma_i(z))=b_\trop(x)$, ce qui entraîne que 
$\sigma_i(z)\in Y_x$ puisque~$b_\trop$ est injective sur~$C$. 
Et l'on a~$g_\trop(\sigma_i(z))=g_\trop(y)$, ce qui entraîne que~
$\sigma_i(z)=y$ car~$g_\trop$ est injective sur~$\Sigma_{g|_{Y_x}}=\Sigma$. 
En conséquence~$y\in \Tau$, ce qui achève la démonstration. 
\end{proof}

\begin{prop}\label{prop-polydisque-riche}
Soit~$n$ un entier
et  soit~$r=(r_1,\ldots, r_n)$ un polyrayon constitué 
d'éléments de~$\abs{k^\times}^\Q$.
Le polydisque fermé de polyrayon~$r$ est très riche.
\end{prop}
 
 \begin{proof}
Soit $D$ ce polydisque.
Il est riche (théorème~\ref{theo.riche}). 
On raisonne par récurrence forte sur~$n$. 
Soit~$\Sigma$ une partie paralinéaire compacte de~$D$. Il s'agit de démontrer que~$\Sigma$ est contenue
dans une partie paralinéaire compacte de~$D$ purement de dimension~$n$. 
Soit~$x$ un point de~$X$. Il existe un voisinage affinoïde~$V$ de~$x$ dans~$X$
et un moment~$f\colon V\to \gm^n$ tel que~$\Sigma\cap V$ soit contenu dans~$\Sigma_f$
(proposition \ref{prop.parties-polyedrales-moments}). 
Par compacité de~$\Sigma$, 
il suffit d'exhiber une partie paralinéaire~$\Tau$ de~$X$ purement de dimension~$n$ contenant un voisinage
de~$x$ dans~$\Sigma_f$. 
Notons~$p_1,\ldots, p_n$ les~$n$ projections de~$D$ 
sur les disques de rayons~$r_1,\ldots, r_n$.
Le bord~$\partial D$ est alors la réunion des~$p_i^{-1}(\eta_{r_i})$, 
et la trace de~$p_i^{-1}(\eta_{r_i})$
chacun d'eux sur~$\Sigma_f$ est pour tout $i$ une partie
paralinéaire de ce dernier de dimension~$\leq n-1$, 
et c'est aussi une partie squelettique de $p_i^{-1}(\eta_{r_i})$
(lemme \ref{lemme-fibreetar-pl}) ; de plus, $\partial V\cap \Sigma_f$
est lui aussi paralinéaire de dimension~$\leq n-1$ (corollaire \ref{coro-bor-petit}). 
Fixons une décomposition cellulaire~$\mathscr C$ de~$\Sigma_f$ adaptée à son intersection avec~$\partial V$ et avec chacune
des~$p_i^{-1}(\eta_{r_i})$. Comme~$V$ est un voisinage de~$x$, il suffit de montrer que chacune des cellules~$C$ de~$\mathscr C$
contenant~$x$ est contenue dans une partie squelettique compacte de~$X$ purement de dimension~$n$. 

Soit donc~$C$ une telle cellule. 

Supposons tout d'abord que~$C$ n'est contenue dans aucun des~$p_i^{-1}(\eta_{r_i})$. 
La cellule ouverte~$\mathring C$ est alors contenue dans~$\Int(X)$ 
et comme elle rencontre~$\Int(V/X)$ puisque ce dernier contient~$x$, 
elle rencontre~$\Int(V)$. 
Comme~$\Sigma_f$ est purement de dimension~$n$ en tout point de~$\Int(V)$, 
la cellule~$C$ est contenue dans une~$n$-cellule
de~$\mathscr C$. 

Supposons maintenant qu'il existe~$i$ tel que~$C$ soit contenue dans~$p_i^{-1}(\eta_{r_i})$. Cette fibre étant un polydisque de
polyrayon~$(r_j)_{j\neq i}$, l'hypothèse de récurrence assure que~$C$ est contenue dans une partie squelettique
compacte~$C'$ de~$p_i^{-1}(\eta_{r_i})$ purement de dimension~$n-1$. 
Puisque $r_i\in \abs{k^\times}^\Q$,
la proposition~\ref{prop-extension-squelette} 
garantit que~$C'$ est contenue dans une 
partie squelettique compacte de~$D$ purement de dimension~$n$, 
ce qui conclut la démonstration.
\end{proof}
 
\begin{theo}\label{theo.strict-tres-riche}
Soit~$X$ un espace~$k$-analytique purement de dimension~$n$. 

\begin{enumerate}
\item
Soit $x$ un point de~$X$ tel que le germe $(X,x)$  soit strict.
Soit $\Sigma$ une partie squelettique de~$X$ contenant~$x$.
Il existe une partie squelettique de~$X$, purement de dimension~$n$,
qui contient un voisinage de~$x$ dans~$\Sigma$.

\item
Si l'espace~$X$ est strict, il est très riche.
\end{enumerate}
\end{theo}
\begin{proof}
Il suffit de démontrer la première assertion.
On sait déjà que~$X$ est riche (théorème~\ref{theo.riche}). 
Il existe un voisinage de~$x$ dans~$X$ qui est réunion finie
de domaines affinoïdes dont le germe en~$x$ est strict. 
Comme l'assertion est G-locale sur le germe~$(X,x)$,
on peut supposer que $X$ est affinoïde.
Notons~$A$ l'algèbre des fonctions analytiques sur~$X$. 
Utilisons le symbole~$\widetilde \cdot$ plutôt que~$\widetilde \cdot^1$ pour les réductions strictes, 
 qui sont les seules dont nous aurons besoin puisque le germe considéré est strict. 

 Le germe~$\widetilde{(X,x)}$ est égal à~$\zr{\hrt x}{\widetilde k}\{R\}$ où~$R$ est la~$\widetilde k$-algèbre de type 
 fini image de~$\widetilde A$ dans~$\hrt x$. 
D'après le lemme de normalisation de Noether appliqué à la
 $\widetilde k$-algèbre~$R$,  il existe une famille
 finie de fonctions~$f_1,\ldots, f_d$ sur~$X$ de normes spectrales
égales à~$1$ telles que~$\abs{f_i(x)}=1$ pour tout~$i$, 
telles que les~$\widetilde{f_i(x)}$ 
soient algébriquement indépendantes sur~$\widetilde k$, et telles que
l'algèbre~$R$ soit entière sur~$\widetilde k[\widetilde{f_i(x)}]_i$.

Comme un anneau de valuation est intégralement clos,
cette dernière condition entraîne que~$\zr{\hrt x}{\widetilde k}\{R\}=\zr{\hrt x}{\widetilde k}\{\widetilde{f_i(x)}\}$. 
 La famille~$(f_1,\ldots, f_d)$ induit alors un morphisme~$f\colon X\to D$, 
où $D$~est le $d$-polydisque unité, car les normes spectrales
des~$f_i$ sont égales à~$1$.
Comme $\abs{f_i(x)}=1$ et comme les $\widetilde{f_i(x)}$
sont algébriquement indépendants, on a $f(x)=\eta_1$.
Comme $(\widetilde{X,x})= \zr{\hrt x}{\widetilde k}\{\widetilde{f_i(x)}\}$,
la réduction du germe $(\widetilde{X,x})$ est égale à l'image
réciproque de la réduction du germe de~$D$ en~$\eta_1$;
cela démontre que le morphisme~$f$ est intérieur en~$x$.
D'après le théorème~10.3.7 de~\cite{ducros2018}, $f$ est plat au-dessus
de~$\eta_1$, et en particulier en~$x$. Soit $U$ l'ensemble des points
de~$\Int(X)$ en lesquels $f$ est plat; c'est un voisinage ouvert de~$x$.

D'après la proposition \ref{prop-polydisque-riche}, 
le disque unité~$D$ est très riche.
Le théorème~\ref{theo-tresriche-relatif} assure que~$U$ est très riche. 
Il existe donc une partie squelettique compacte~$\Tau$ de~$U$,
purement de dimension~$n$,
et contenant un voisinage de~$x$ dans~$\Sigma$. 
 \end{proof}

\section{Formule de projection}

\subsection{}
Nous reprenons le contexte du paragraphe précédent:
$p\colon Y \to X$ est un morphisme d'espaces analytiques équidimensionnels
et topologiquement séparés;
on pose $\dim(X)=m$, $\dim(Y)=n$ et $d=n-m$.
On considère~$X$ comme espace tropical de niveau~$m$
et~$Y$ comme espace tropical de niveau~$n$.

\begin{rema}\label{rema.dim-fibres}
Notons~$U$ le lieu de platitude de~$p$; c'est un ouvert de Zariski de~$Y$
(\cite{ducros2018}, théorème~10.3.2).
Notons aussi~$V$ l'ensemble des points~$y\in Y$
en lesquels la dimension relative de~$p$ prend sa valeur minimale~$\delta$;
c'est également un ouvert de Zariski de~$Y$
(\cite{ducros2007}, th. 4.9). 

\begin{enumerate}
\item
Soit $y$ un point de~$V$ tel que $d_k(y)=n$; 
il en existe car $V$ n'est pas vide.
Alors, 
$n=d_k(y)=d_k(p(y))+d_{\hr{p(y)}}(y) \leq m+\delta$.
Par suite, $d\leq \delta$.

\item
Soit $y$ un point de~$U$. 
D'après le lemme~4.5.11 de~\cite{ducros2018}, 
la dimension relative de~$p$ en~$y$ est alors égale à~$d$.
Cela prouve que $\delta=d$ si $U$ n'est pas vide;
on a alors $U\subset V$.

\item\label{rema.dim-fibres,2}
Soit $x\in X$ tel que $d_k(x)=m$ et soit $y\in Y_x$.
On a $d_{\hr{p(y)}}(y)=d_k(y)-d_{k}(x) \leq n-m=d$; 
par suite $\dim(Y_x)\leq d$.
Puisque $d$ est la valeur minimale de la dimension relative,
cela prouve que $Y_x$ est purement de dimension~$d$;
on a ainsi $Y_x\subset V$.
Cela entraîne également $\delta=d$.

\item
On suppose que $X$ est réduit, que $\delta=d$ et que $V$ est dense dans~$Y$.
Démontrons alors que $U$ est dense.

Il suffit de démontrer que $U$ contient tout tout point~$y$
de~$V$ tel que $d_k(y)=n$.
Posons $x=p(y)$. Comme les fibres de~$p|_V$ sont de dimension~$d$, on a
$d_k(x)=n-d=m$ (\cite{ducros2018}, 1.4.14 (4)). 
L'espace $X$ étant réduit, le morphisme~$p$ est alors plat en~$y$ 
par le théorème~10.3.7 de \cite{ducros2018}, c'est-à-dire $y\in U$. 
\end{enumerate}
\end{rema}

\subsection{}
Soit $\omega$ une G-forme de type~$(d,d)$ sur~$Y$
dont le support est relativement propre sur~$X$.
Au paragraphe précédent, nous avons défini l'intégrale~$u_\omega(x)$
de~$\omega$ dans les fibres de~$p$ en un point squelettique~$x$ de~$X$,
ou bien un point~$x$ quelconque lorsque $p$ est plat;
dans ces deux cas, la fibre~$Y_x$ est purement de dimension~$d$.

Dans ce paragraphe,  nous allons faire l'hypothèse 
que le lieu de platitude de~$p$ est dense dans~$Y$.
Pour tout $x\in X$ tel que $Y_x$ est de dimension~$d$, on pose $u_\omega(x)
= \int_{Y_x}\omega$;
sinon, on pose $u_\omega(x)=0$.

\begin{theo}\label{theo.integrale-fibre}
On suppose que le lieu de platitude de~$p$ est dense dans~$Y$.
Soit $\omega$ une G-forme de type~$(d,d)$ sur~$Y$
à coefficients tropicalement boréliens localement bornés
et dont le support est $p$-compact.
La fonction~$u_\omega$ définie comme ci-dessus
est localement intégrable sur~$X$ et l'on a, 
pour toute $(m,m)$-forme lisse~$\alpha$ à support compact sur~$X$, l'égalité
\[ \int_Y \omega\wedge p^*\alpha = \int_X u_\omega \,\alpha_\Rad. \]
\end{theo}

\begin{coro}
On a l'égalité de courants
$p_*[\omega] = [u_\omega]$.
\end{coro}

\begin{rema}
Il résulte de la remarque~\ref{rema.dim-fibres}
que les hypothèses  du théorème sont satisfaites
lorsque $X$ est réduit, $d=\delta$ et $V$ est dense dans~$Y$.

L'énoncé ne vaut pas en général sans hypothèse de  platitude sur~$p$.
Lorsque $X$ n'est pas réduit, la formule est fausse si 
$p$ est l'immersion (non plate) de $X_\red$ dans~$X$, 
son défaut étant mesuré
par les multiplicités génériques de~$X$.
\end{rema}

\begin{proof}
Considérons une $(m,m)$-forme lisse~$\alpha$ à support compact sur~$X$;
il s'agit de montrer que $u_\omega$ est intégrable pour la mesure~$\alpha_\Rad$
et de prouver l'égalité 
\[ \int_Y \omega\wedge p^*\alpha = \int_X u_\omega \,\alpha_\Rad. \]
Pour cela, on peut remplacer~$X$ par un voisinage analytique compact
du support de~$\alpha$
 puis, le support de~$\omega$ étant alors compact,
$Y$ par un voisinage analytique compact de ce support. 
Nous supposons donc que $X$ et $Y$ sont compacts.

Le lieu de~$Y$ constitué des points 
en lesquels $p$ est dimension relative~$>d$ est un fermé de Zariski de~$Y$,
donc son image dans~$X$ est compacte.
D'après la remarque~\ref{rema.dim-fibres}, \ref{rema.dim-fibres,2}), cette image est disjointe du support de~$\alpha$.
Quitte à remplacer de nouveau~$X$ par un voisinage analytique compact
du support de~$\alpha$, on suppose que les fibres de~$p$
sont toutes purement de dimension~$d$.

L'espace~$X$ possède un G-recouvrement fini par des domaines analytiques
compacts tropicalisant~$\alpha$;
par additivité des intégrales en jeu, on peut donc supposer que 
la forme~$\alpha$ sur~$X$ est tropicale.
De même, l'espace~$Y$ possède un G-recouvrement fini par des domaines
analytiques compacts tropicalisant~$\omega$;
par additivité, on peut donc de même supposer que la forme~$\omega$
est tropicale sur~$Y$.

Soit $f\colon X\to T$ un moment tropicalisant~$\alpha$
et $h=(f\circ p,g)\colon Y\to T\times T'$ 
un moment tropicalisant~$\omega$ qui  prolonge~$f\circ p$.
En raisonnant G-localement sur~$X$ et sur~$Y$,
on peut supposer qu'il existe une cellule~$C$ de~$T_\trop$
et une cellule~$D$ de~$(T\times T')_\trop $
telles que $f_\trop(X)=C$ et $h_\trop(Y)=D$.
L'image de~$D$ dans~$C$ par la projection de $(T\times T')_\trop$
sur~$T_\trop$ est une cellule~$C'$ de~$T_\trop$ contenue dans~$C$.
On a $C'=(f\circ p)_\trop(Y)$.
Distinguons deux cas:

\emph{Supposons que $\dim(C')<m$.}
Dans ce cas, le morphisme~$p$ se factorise par le domaine analytique
$f_\trop^{-1}(C')$ sur lequel la forme~$\alpha$ est nulle,
ainsi que la mesure~$\alpha_\Rad$.
Comme $u$ est nulle en dehors de~$p(Y)$, l'égalité voulue est immédiate.

\emph{Supposons que $\dim(C')=m$.}
En commençant par extraire $m$ fonctions convenables du moment~$f$,
on suppose d'abord que $T=\gm^m$; en extrayant $\dim(D)-m$
fonctions du moment~$g$ et en complétant au besoin 
par la fonction constante~$1$ (répétée $n-\dim(D)$ fois), on suppose 
que $T'=\gm^d$.

D'après 
l'assertion~\emph a) du lemme~\ref{lemm.dimensions-tropicales-fibre},
on a $p(\Sigma_h)\subset\Sigma_f$ et le morphisme~$p$ 
induit par restriction une application paralinéaire
de~$\Sigma_h$ dans~$\Sigma_f$.

Soit $\mathscr C$ une décomposition cellulaire de~$\Sigma_f$ 
en $f$-cellules commodes 
et une décomposition cellulaire de~$\Sigma_h$ 
en $h$-cellules commodes (lemme~\ref{lemm.exists-dec-cell-commode}). 
Quitte à raffiner ces décompositions cellulaires, on suppose
que l'image d'une cellule de~$\mathscr D$ est une cellule de~$\mathscr D$.

On note $\mathscr C_m$ l'ensemble des cellules de dimension~$m$
de~$\mathscr C$ et $\mathscr D_n$ l'ensemble des cellules de dimension~$n$ de~$\mathscr D$.

Dans la suite du calcul, nous identifierons toute cellule de~$\mathscr C$
à son image par~$f_\trop$ dans~$(\gm^m)_\trop$ 
et toute cellule de~$\mathscr D$ à son image par~$h_\trop$ dans~$(\gm^{m+d})_\trop$. De même, les fibres $D\cap p^{-1}(x)$ sont identifiées à leur image
par $g_\trop$ dans~$(\gm^d)_\trop$.
Lorsqu'elles sont de dimension~$m$ (resp.~$n$, resp.~$d$), 
on les munit du calibrage déduit de cette identification;
autrement dit, leur calibrage canonique est ce calibrage multiplié 
par un degré convenable. 

%
%
 
 
Pour toute $n$-cellule~$D\in\mathscr D_n$,
il existe un unique entier~$d_D$ tel que
le morphisme~$h$ soit fini et plat de degré~$d_D$ 
en tout point de~$\mathring D$. 
On a ainsi
\[ \int_Y \omega\wedge p^*\alpha = \sum_{D\in\mathscr D_n}
 d_D    \int_{D} \omega\wedge p^*\alpha.  \]

Fixons une telle cellule~$D$ et posons $C=p(D)$.
Comme les fibres de~$p$ sont de dimension~$\leq d$, on a $\dim(C)\geq m$,
donc $\dim(C)=m$.
On applique alors le théorème de Fubini: pour presque tout
$x\in C$, l'intégrale 
$ \int_{D\cap p^{-1}(x)} \omega $ est convergente,
la fonction $u_D\colon x\mapsto \int_{D\cap p^{-1}(x)}\omega$
est intégrable sur~$C$ et l'on a 
\[ \int_D \omega\wedge p^*\alpha = \int_C u_D \alpha. \]
(Lorsque $\omega$ est de classe~$\mathscr C^p$,
il découle de la proposition~\ref{prop.continuite-integrale}
que la fonction~$u_D$ est même G-localement de classe~$\mathscr C^p$ sur~$C$.)
Il existe un unique entier~$d_C$ tel que 
le morphisme~$f$ soit fini et plat de degré~$d_C$ 
en tout point de~$\mathring C$.
Ainsi,
\[ \int_Y \omega\wedge p^*\alpha = 
\sum_{C\in\mathscr C_m} d_C \int_C 
\left(\sum_{\substack{D\in\mathscr D_n\\ p(D)=C}} (d_D/d_C) u_D \right) 
\alpha. \]

Fixons maintenant une cellule~$C\in\mathscr C_m$.
Soit $x\in\mathring C$ tel que $u_D$ soit définie en~$x$ pour tout~$D\in\mathscr D_n$ tel que $p(D)=C$; cela n'exclut qu'un ensemble de mesure nulle de~$C$.
La famille $(D\cap p^{-1}(x))$, pour $D\in\mathscr D$ tel que $p(D)=C$,
est une décomposition
cellulaire de~$\Sigma_h\cap p^{-1}(x)$, qui est égal à~$\Sigma_{h_x}$
d'après le lemme~\ref{lemm.dimensions-tropicales-fibre}.
Soit $D$ une telle cellule.
Lorsque $\dim(D)<n$, on a $\dim(D\cap p^{-1}(x))<d$
et cette cellule n'intervient pas dans le calcul de $\int_{Y_x}\omega$.
Supposons donc que $\dim(D)=n$. Nécessairement, tout point
de~$\mathring D$ est intérieur à~$Y$, par définition d'une cellule commode.
Puisque le lieu de platitude de~$p$ est ouvert de Zariski qui est dense, 
il contient tous les points~$y\in Y$ tels que $d_k(y)=n$, et en particulier~$\Sigma_h$.
Le lemme~\ref{lemm.dimensions-tropicales-fibre} entraîne alors
que $g|_{Y_x}$ est fini et plat en tout point de~$\mathring D\cap p^{-1}(x)$
et que son degré est égal à $d_D/d_C$.
Cela prouve que l'intégrale $u_\omega(x)=\int_{Y_x}\omega$ converge et que 
\[
u_\omega(x)=\int_{Y_x} \omega 
 = \sum_{\substack{D\in\mathscr D_n \\ p(D)=C}} 
(d_D/d_C) \int_{ D \cap p^{-1}(x)}  \omega 
= \sum_{\substack{D\in\mathscr D_n \\ p(D)=C}} 
(d_D/d_C) u_D(x).
\]
De plus, la fonction~$u_\omega$ définie presque partout sur~$C$ est intégrable
et l'on a 
\[ d_C \int_C  u_\omega \alpha 
 = d_C \sum_{\substack{D\in\mathscr D_n \\ p(D)=C}} 
(d_D/d_C) \int_{ D }  \omega \wedge p^*\alpha
= \sum_{\substack{D\in\mathscr D_n \\ p(D)=C}} 
d_D \int_D \omega\wedge p^*\alpha. \]

Cela démontre que la fonction~$u_\omega$ est localement intégrable sur~$X$
et l'on a 
\begin{align*}
 \int_X u_\omega \alpha_\Rad
& = \sum_{C\in\mathscr C_m} d_C \int_C  u \alpha 
= \sum_{\substack{D\in\mathscr D_n \\ \dim(p(D))=m}} 
d_D \int_D \omega\wedge p^*\alpha  \\
&= \sum_{\substack{D\in\mathscr D_n }}
d_D \int_D \omega\wedge p^*\alpha
= \int_Y \omega\wedge p^*\alpha. \end{align*}
Cela termine la démonstration.
\end{proof}

\begin{theo}\label{theo.formule-projection}
Soit $p\colon Y\to X$ un morphisme d'espaces analytiques 
topologiquement séparés, équidimensionnels
et sans bord;
posons $m=\dim(X)$, $n=\dim(Y)$  et $d=n-m$.
On suppose que le lieu de platitude de~$p$ est dense dans~$Y$.
Soit $u_1,\dots,u_r$ des fonctions 
tropicalement d-convexes sur~$X$,
soit $\omega$ une forme G-lisse de type~$(d,d)$ sur~$Y$
dont le support est $p$-compact 
et soit $u_\omega\colon X\to\R$ la fonction 
définie par l'intégrale 
\[ u_\omega(x) = \int_{Y_x} \omega \]
de la forme de type~$(d,d)$ induite par~$\omega$ sur l'espace
$\hr x$-analytique~$Y_x$ si cet espace est de dimension~$d$,
et par $u_\omega(x)=0$ sinon.
Soit $\alpha$ une $(m-r,m-r)$-forme lisse à support compact sur~$X$.

Les fonctions $p^*u_1,\dots,p^*u_r$ sur~$Y$ sont tropicalement d-convexes
et l'on a  
\[ 
\langle \ddc(p^* u_1)\wedge\dots\wedge \ddc(p^* u_r),
	\omega\wedge p^*\alpha\rangle
= 
\langle \ddc u_1 \wedge\dots\wedge \ddc u_r, 
u_\omega\cdot\alpha \rangle. \]
\end{theo}

Pour justifier l'existence du membre de gauche
de cette dernière formule, notons que 
le support de $\omega\wedge p^*\alpha$ est compact dans~$Y$.
D'autre part, l'application 
$\phi\mapsto \langle \ddc u_1 \wedge\dots\wedge \ddc u_r, 
\phi\alpha\rangle $ 
est une mesure de Radon sur~$X$ dont le support
est contenu dans une partie squelettique compacte~$\Sigma$
(remarque~\ref{rema.support-bt-squelette})
et le membre de droite de la formule 
est obtenu
en intégrant la fonction $u_\omega|_\Sigma$ contre cette mesure;
observons que 
$u_\omega|_\Sigma$ est G-lisse car 
$u_\omega$ est G-lisse (théorème~\ref{theo-variation-integrales}).

\begin{proof}
Le théorème est de nature locale sur~$X$. 
Par multilinéarité, on peut alors supposer 
qu'il existe un moment $f\colon X\to T$
et des fonctions convexes $v_1,\dots,v_r$ sur un voisinage
polyédral convexe compact~$P$ de~$f_\trop(X)$ dans~$T_\trop$ 
telles que $u_j=f^*v_j$ pour tout~$j$.
Quitte à étendre le moment~$f$,
on peut enfin supposer que la forme~$\alpha$ s'écrit
$\rho f^*\eta$, où $\eta$ est une forme lisse sur~$T_\trop$
et $\rho $ une fonction lisse à support compact sur~$X$.

Il existe des familles~$(v_j^\eps)$ de fonctions lisses convexes
sur~$P$ qui convergent uniformément vers les~$v_j$.
Les $u_j^\eps=f^*v_j^\eps$ sont psh lisses et convergent uniformément
vers~$u_j$; alors le courant
\[ \ddc(u_1)\wedge \dots\wedge \ddc(u_r) \]
est limite des courants associés aux formes lisses 
$\ddc(u_1^\eps)\wedge \dots\wedge \ddc(u_r^\eps)$,
et la mesure 
\[ \phi \mapsto \langle \ddc(u_1)\wedge \dots\wedge \ddc(u_r) ,
\phi \alpha \rangle \]
est limite des mesures
\[ \phi \mapsto \langle \ddc(u_1^\eps)\wedge \dots\wedge \ddc(u_r^\eps) ,
\phi \alpha \rangle .\]
Ces mesures sont toutes supportées par~$\Sigma_f$.
La fonction~$u_\omega|_{\Sigma_f}$ est $\mathscr C^\infty$ par morceaux
(théorème~\ref{theo-variation-integrales}),
si bien que
$\langle \ddc(u_1)\wedge \dots\wedge \ddc(u_r) ,
u_\omega  \alpha \rangle $
est la limite des intégrales
\[  \langle \ddc(u_1^\eps)\wedge \dots\wedge \ddc(u_r^\eps) ,
u_\omega \alpha \rangle. \]

D'après le théorème~\ref{theo.integrale-fibre},
on  a
\[ \langle \ddc(u_1^\eps)\wedge \dots\wedge \ddc(u_r^\eps) ,
u_\omega \alpha \rangle
 = \langle \ddc (p^*u_1^\eps) \wedge \dots \wedge \ddc(p^*u_r^\eps),
\omega \wedge p^*\alpha \rangle
\]
pour tout~$\eps$,
qui convergent
vers 
\[ \langle \ddc (p^*u_1) \wedge \dots \wedge \ddc(p^*u_r),
\omega \wedge p^*\alpha \rangle ,\]
d'où le théorème.
\end{proof}

\begin{coro}
On a l'égalité de courants 
\[ p_*(
\ddc (p^* u_1)\wedge\dots\wedge \ddc (p^* u_r)  \wedge\omega)
= 
\ddc u_1 \wedge\dots\wedge \ddc u_r \cdot u_\omega. \]
\end{coro}

\section{Courants et cycles}

Soit $X$ un $n$-espace $k$-analytique, considéré comme espace tropical
de niveau~$n$.

\subsection{}
Soit $Y$ un sous-espace analytique fermé de~$X$, purement de dimension~$p$. 
On le considère comme un espace tropical de niveau~$p$
et on note $j\colon Y\to X$ l'immersion fermée
canonique. 
Le courant $j_*\delta_Y$ sur~$X$ est de bidimension~$(p,p)$;
on le note plus simplement~$\delta_Y$.

\subsection{}
Soit $p$ un entier et soit $z=\sum z(Y)Y$ un cycle purement $p$-dimensionnel
sur~$X$. Soit $\omega\in\mathscr A^{p,p}_\cpct(X)$; posons 
\[ \langle \delta_z,\omega\rangle = \int_z \omega 
= \sum_Y z(Y) \int_Y \omega. \]
C'est une somme finie: en effet, comme $\omega$ est à support
compact et la famille $(z(Y))$ est localement finie, 
l'ensemble des fermés de Zariski intègres~$Y$ de~$X$
tels que $z(Y)\neq 0$ et $\omega|_Y\neq 0$ est fini.
L'application $\omega\mapsto \int_z\omega$ est un courant:
si l'on se restreint aux formes à support
dans un domaine analytique compact donné~$V$, il s'agit de la somme
finie des courants $z(Y) \delta_Y$, où $Y$ parcourt l'ensemble
fini des fermés de Zariski intègres~$Y$ de~$X$ qui rencontrent~$V$
et tels que $z(Y)\neq 0$.


Lorsqu'on se donne un sous-espace analytique fermé de~$X$,
purement de dimension~$p$,
on peut considérer d'une part l'intégration sur ce sous-espace
et d'autre part l'intégration sur le cycle associé.
La proposition suivante affirme que ces deux opérations coïncident.

\begin{prop}\label{prop.cycle-courant}
Soit $Y$ un sous-espace analytique fermé de~$X$, purement de dimension~$p$.
On a $\delta_Y=\delta_{[Y]}$.
\end{prop}
\begin{proof}
Notons $j$ l'immersion de~$Y$ dans~$X$.
Soit $\omega$ une $(p,p)$-forme à support compact sur~$X$.
On a donc $\langle\delta_Y,\omega\rangle = \int_Y j^*\omega$,
ce qui permet de supposer que $Y=X$.

Soit $(X_i)$ la famille des composantes irréductibles de~$X$ et,
pour tout~$i$, soit $m_i$ la multiplicité de~$X$ le long de~$X_i$.
Soit $X'$ l'ouvert de Zariski de~$X$ constitué des points
qui sont situés sur une unique composante irréductible de~$X$;
il contient le support de~$\omega$. Pour tout~$i$,
posons $X'_i=X'\cap X_i$; c'est un ouvert de~$X$ et l'on a
\[ \int_{X_{i,\red}}\omega
 = \int_{X'_{i,\red}}\omega
= \frac1{m_i} \int_{X'_{i}}\omega, \]
en vertu de la proposition~\ref{prop.mesure-f*-degre-integral} appliquée
à l'immersion de~$X'_{i,\red}$ dans~$X'_i$, dont le
degré calibré est~$1/m_i$ (proposition~\ref{exemple-degcal-fini}).
Alors, on a 
\[ \langle \delta_X,\omega\rangle = \int_X\omega=\int_{X'}\omega
= \sum_i \int_{X'_i}\omega
= \sum_i m_i \int_{X'_{i,\red}} \omega 
= \sum_i m_i \int_{X_{i,\red}} \omega 
= \langle \delta_{[X]}, \omega\rangle, \]
comme il fallait démontrer.
\end{proof}

\subsection{}
Si $z=\sum z(Y)Y$, on définit aussi $\delta^\partial_{z}
=\sum z(Y) \delta^\partial_{Y}$. 
Si $Y$ est un sous-espace analytique fermé de~$X$, on démontre
comme pour la proposition~\ref{prop.cycle-courant}
que $\delta^\partial_{Y}=\delta^\partial_{[Y]}$.

\section{La formule de Poincaré--Lelong}\label{poinclel} 

\subsection{}
Soit~$X$ un $n$-espace~$k$-analytique sans bord; 
on le considère comme un espace tropical de niveau~$n$.

Soit $f$ une  fonction méromorphe régulière~$f$ sur~$X$.
Le courant d'intégration $\ddiv f$ sur le cycle~$\div(f)$
est un courant symétrique de bidimension~$(n-1,n-1)$ sur~$X$,
donc de bidegré~$(1,1)$; il est $\di$-fermé  et $\dc$-fermé
(proposition~\ref{prop.deltaY-ferme}).

La formule de Poincaré--Lelong (théorème~\ref{pcrllg})
va entraîner qu'il est $\ddc$-exact.

\subsection{}
Soit~$f$ une fonction méromorphe régulière sur~$X$.
Soit~$U$ le sous-espace de~$X$ formé des points au voisinage
desquels~$f$ est définie et non nulle ;
c'est un ouvert de Zariski de~$X$
sur lequel  la fonction~$\log\,\abs f$ est lisse,
et il est Zariski dense.
Pour tout point~$x$ du fermé complémentaire, on  a $d_k(x)<\dim(X)$.
D'après le paragraphe~\ref{ss.prolongement-courant}, 
le courant $[\log\abs f]$ sur~$U$ se prolonge un unique courant
de bidegré~$(0,0)$ sur~$X$ que nous noterons encore $[\log \abs f]$.

\begin{theo}[Formule de Poincaré--Lelong]
\label{pcrllg}\label{theo.poincare-lelong}
Soit~$X$ un $n$-espace~$k$-analytique sans bord
et soit~$f$ une fonction méromorphe régulière sur~$X$. 
On a l'égalité \[\ddiv f=\ddc[\log \abs f].\]
\end{theo} 
L'exemple~\ref{exem.f*ddcf*u} 
montre que l'hypothèse de régularité sur~$f$ est nécessaire,
même si le complémentaire de son lieu d'inversibilité est de dimension~$<n$.

\begin{lemm}\label{lemme-pl-inversible}
L'égalité ci-dessus est vraie sur l'ouvert d'inversibilité de $f$.
\end{lemm}

\begin{proof}
On peut supposer $f$ inversible. On a alors $\div f=0$ et le terme de gauche
de l'égalité est donc nul. 
On a par ailleurs $\di \dc x=0$ dans $\mathscr A^{1,1}(\R)$ ; il vient~$\di \dc \log \abs f=0$.
\end{proof}
 
\begin{lemm}
Il suffit de démontrer le théorème~\ref{theo.poincare-lelong}
si l'on suppose de plus que 
$X$~satisfait à la propriété~$(\mathrm S_1)$
et que $f$~est une fonction holomorphe régulière qui définit
un morphisme purement de dimension relative~$n-1$
de~$X$ vers~$\A^1$.
\end{lemm}
 \begin{proof} 
Le théorème est local sur~$X$; il suffit donc de prouver que tout
point~$x$ de~$X$ a un voisinage sur lequel cette égalité de courant
est vérifiée.
On peut ainsi commencer par supposer 
que $f$ est quotient de deux fonctions holomorphes régulières.
Par additivité des deux termes de l'égalité à établir, on peut
donc supposer que $f$ est holomorphe.
Pour alléger les notations, nous noterons encore~$f$ le morphisme 
$X\to \A^1$ qu'elle induit. 

Si~$f(x)\neq 0$, il existe un voisinage~$U$ de~$x$ sur
lequel~$f$ est inversible, et l'égalité à établir vaut sur $U$ en vertu du lemme
\ref{lemme-pl-inversible}.

Supposons que~$f(x)=0$.  
Comme la fonction~$f$ est régulière,
le morphisme~$f$ qu'elle induit est de dimension relative~$n-1$ en~$x$.
Par suite, ce morphisme est de dimension~$\leq n-1$ au voisinage de~$x$
(\cite{ducros2007}, th. 4.9). 
En vertu du \emph{Hauptidealsatz}
(lemme~\ref{lemm.hauptidealsatz-anal}), le morphisme~$f$ est donc
purement de dimension relative~$n-1$ au voisinage de~$x$.

Démontrons que l'espace~$X$ a la propriété~$(\mathrm S_1)$ en~$x$.
Soit $V$ un domaine  affinoïde de~$X$ contenant~$x$.
Le germe de~$f|_V$ en~$x$ est un élément de l'idéal maximal
de~$\mathscr O_{V,x}$ et est un élément régulier de cet anneau.
En conséquence, l'anneau local~$\mathscr O_{V,x}$ 
a la propriété~$(\mathrm S_1)$, 
si bien que $X$ a la propriété~$(\mathrm S_1)$ en~$x$. 
Il existe dès lors un voisinage ouvert~$U$ de~$x$ dans~$X$ 
tel que~$U$ satisfasse la propriété~$(\mathrm S_1)$ 
et tel que~$f|_U$ soit purement de dimension relative~$n-1$. 
\end{proof}

Le lemme
précédent permet de supposer à partir de maintenant que
$X$ a la propriété~$(\mathrm S_1)$ et que $f$ est une fonction
holomorphe régulière induisant un morphisme purement
de dimension relative~$n-1$ de~$X$ dans~$\A^1$.

Les deux courants~$\delta_{\div f}$
et~$\ddc [\log \abs f]$ sont symétriques ; pour
s'assurer de leur égalité, 
il suffit de les appliquer à une forme
test symétrique.

Soit donc~$\omega$ une forme lisse de type~$(n-1,n-1)$ sur~$X$,
symétrique et à support compact.
%
Comme le support de~$\omega$ est compact,
il existe une famille finie $(V_i)_{1\leq i\leq m}$
de domaines analytiques compacts de~$X$  tropicalisant~$\omega$ et 
telle
tout point du support de~$\omega$ soit contenu dans l'intérieur 
de l'un des~$V_i$.
Notons~$V$ la réunion des~$V_i$. C'est un domaine analytique compact 
de~$X$. Par construction, le support de~$\omega$ est contenu dans 
l'intérieur topologique~$\mathring V$ de~$V$ dans~$X$,
qui coïncide avec~$\Int(V/X)$. 

Pour toute partie non vide~$E$ de~$\{1,\ldots, m\}$ et tout $s>0$, 
on pose $V_E=\bigcap_{i\in E}V_i$,
$V_0=V\cap \{f=0\}$, $V_{\geq s}=V\cap \{\abs f\geq s\}$, 
$V_{E,0}=V_E\cap V_0$ et $V_{E,\geq s}=V_E\cap V_{\geq s}$.
L'espace analytique~$V_E$ est intersection d'une famille non vide
de domaines analytiques compacts de~$X$ qui tropicalisent~$\omega$;
il possède en particulier la propriété~(S\textsubscript1).
De plus $\omega$ est tropicale sur~$V_E$ ; 
soit $g_E\colon V_E\to T$ un moment et soit $\alpha_E$ une $(n-1,n-1)$-forme
sur~$g_{E,\trop}(V_E)$ telle que $\omega|_{V_E}=g^*\alpha_E$;
posons aussi $h_E=(g_E,f)\colon V_E\to T\times\A^1$.
Choisissons un nombre réel~$r>0$ pour lequel
les conclusions du théorème~\ref{theo.tropcont} sont satisfaites 
par chacun des couples $(V_E,g_E)$. 

Soit $(u_1,u_2)$ une partition de l'unité lisse de $\R$, subordonnée
au recouvrement ouvert $(\mathopen]-\infty, \log r\mathclose[, \mathopen]\log r-1,+\infty\mathclose[)$.
Soit $t\in \mathopen]0,r\mathclose[$ tel que le support de $u_1$ soit disjoint
de $\mathopen[\log t, \log r\mathclose[$. 
On note $v_1$ (resp. $v_2$) la fonction lisse sur $X$ égale à $u_1(\log \abs f)$ (resp. $u_2(\log \abs f)$)
sur l'ouvert d'inversibilité de $f$, et à $1$ (resp. $0$) au voisinage de $\{f=0\}$. Posons
$\omega_1=v_1\omega$ et $\omega_2=v_2\omega$. Il suffit de démontrer
que 
\[
\langle \di\dc \log \abs f, \omega_1\rangle=\int_{\{f=0\}} \omega_1
\;\;\text{et}\;\;
\langle \di\dc \log \abs f, \omega_2\rangle=\int_{\{f=0\}} \omega_2. 
\]

Par construction, le support de $\omega_2$
est contenu dans l'ouvert $\{f\neq 0\}$ sur lequel $\di \dc \log \abs f$ s'annule ; les deux termes 
de la seconde égalité sont donc nuls, 
et il suffit désormais d'établir la première. 

Comme $\dc\omega_1$ est une forme de type~$(n-1,n)$,
son support ne rencontre aucun fermé de Zariski 
de dimension~$<n$ (proposition~\ref{lemm.support-comp-ferme}); 
en particulier,  il ne rencontre pas~$f^{-1}(0)$.
Comme ce support est compact et la fonction~$\abs f$ continue, 
on a  $\inf_{\supp(\dc\omega_1)}(\abs{f})>0$ ; 
soit $s\in \mathopen ]0,t\mathclose [$
tel que $\abs f>s$ en tout point 
du support de~$\dc\omega_1$. 
On suppose aussi que $u_1\equiv 1$ 
dans un voisinage de $\mathopen]-\infty,\log(s)]$.
En particulier, 
le support de $\dc\omega$ est contenu 
dans l'intérieur~$\Int(V_{\geq s}/X)=\Int(V_{\geq s})$
du domaine analytique fermé~$V_{\geq s}$ de~$V$.

Par définition du courant $\log\abs f$
et compte-tenu du fait que $V_{\geq s}$ contient le support 
de $\ddc\omega_1$ on a 
\[ \langle \di\dc \log \abs f, \omega_1\rangle
=\int_{V_{\geq s}} \log \abs f \,\ddc \omega_1.\]
Puisque $\omega_1$ est symétrique, la 
formule de Green (proposition~\ref{prop.stokes})
assure que
\[
\int_{V_{\geq s}} \log \abs f \ddc \omega_1=
\int_{V_{\geq s}} (\ddc \log \abs f)\wedge \omega_1
+\int^\partial_{V_{\geq s}} \log \abs f\dc \omega_1 
- \int^\partial_{V_{\geq s}} (\dc \log \abs f) \wedge \omega_1.
\] 
Comme $f$ est inversible sur~${V_{\geq s}}$, on a
$\ddc\log\abs f=0$ sur~${V_{\geq s}}$;
de plus, le support de~$\dc\omega_1$  
ne rencontre pas~$\partial({V_{\geq s}})$;
les deux premiers termes du membre de droite de l'égalité précédente
sont donc nuls, d'où l'égalité
\[ \int_{X} \log\abs f \ddc\omega_1 
= - \int^\partial_{V_{\geq s}} (\dc\log \abs f)\wedge\omega_1. \]


Comme le support de~$\omega_1|_{X_0}$ est contenu dans~$V_0$
et comme $\omega_1=\omega$ au voisinage de~$V_0$,
on a $\int_{V_0}\omega=\int_{X_0}\omega_1$.
Pour terminer la preuve de la formule de Poincaré--Lelong, il suffit
donc de démontrer que  l'on a 
\[  \int^\partial_{V_{\geq s}} \dc(\log \abs f)\wedge\omega_1
= - \int_{V_0} \omega. \]

La famille $(V_{i,0})$ est un G-recouvrement fini de~$V_0$;
on a donc l'égalité
\[
 \int_{V_0} \omega = \sum_{E\neq \emptyset} (-1)^{\Card(E)+1} \int_{V_{E,0}} \omega .\]
Par ailleurs, la famille $(V_{i,\geq s})$ est un G-recouvrement fini
de~$V_{\geq s}$; la proposition~\ref{prop.compensation}
entraîne alors l'égalité:
\[
 \int^\partial_{V_{\geq s}}  \dc(\log\abs f) \wedge \omega_1
= \sum_{E\neq \emptyset} (-1)^{\Card(E)+1}
\int^\partial_{V_{E,\geq s}}  \dc(\log\abs f)\wedge \omega_1.
\]
Comme il n'y a qu'un nombre fini de telles parties~$E$,
il suffit de démontrer que pour toute partie non vide~$E$ de~$\{1,\dots,m\}$,
on a
\[ 
\int^\partial_{V_{E,\geq s}} \dc(\log\abs f)\wedge \omega_1
= -\int_{V_{E,0}} \omega. \]
Fixons donc une
telle $E$. Posons $V_{E,[s,t]}=V_E\cap \{ s\leq\abs f\leq t\}$;
comme le support de~$\omega_1|_{V_{E,\geq s}}$ est contenu dans l'intérieur
de~$V_{E,[s,t]}$ relativement à $V_{E,\geq s}$, on a
\[ \int^\partial_{V_{E,\geq s}} \dc \log\abs f\wedge \omega_1
= \int^\partial_{V_{E,[ s,t]}} \dc \log\abs f\wedge \omega_1. 
\]
Nous allons calculer cette dernière intégrale par voie tropicale.
Posons $\sigma=\log(s)$ et $\tau=\log(t)$.
La restriction de~$h$ à~$V_{E,[s,t]}$ est un moment;
on a $\omega_1|_{V_{E,[s,t]}}=h_E^*(u_1\alpha_E)$
et $\dc \log\abs f \wedge\omega_1|_{V_{E,[s,t]}}
=h_E^*(\dc t\wedge u_1\alpha_E)$.

Notons $(P,\mu)$ le polyèdre calibré $g_{\trop}(V_{E,0})^{(n-1)}$
et $e$ le vecteur $(0,1)$ de $T_\trop\times\gmtrop$ ; 
notons $(Q,\nu)$ le polyèdre calibré
$h_\trop(V_{E,[s,t]})^{(n)}$. 
Par le choix de~$r$ fait au début de la preuve, on a
$(Q,\nu)=(P\times[\sigma,\tau], \mu\times \lambda)$ où
$\lambda$ désigne le calibrage standard de $[\sigma,\tau]$. 
En conséquence 
$\int^\partial_{V_{E,[ s,t]}} \dc \log\abs f\wedge \omega_1
=\int_Q^\partial 
\langle \dc t\wedge u_1\alpha,\nu\rangle.$

Choisissons une décomposition cellulaire 
$\mathscr C$ de $P$
adaptée à son calibrage. L'ensemble des partie de la forme
$C\times\{\sigma\}$, $C\times\{\tau\}$ et $C\times[\sigma,\tau]$, 
où $C\in \mathscr C$, est une décomposition cellulaire de
$P\times[\sigma \tau]$ adaptée à son calibrage. 
est la réunion de $P\times\{\sigma\}$,
Si $F$ est une $(n-1)$-cellule de $P$
la discordance~$\partial\nu_{F\times\{\sigma\}}$ s'identifie
au produit de $(-\mu_F)$ 
par le vecteur-volume~$\abs e$ ; 
comme $\dc t\wedge\alpha=\alpha\wedge \dc t$
et $u_1(\sigma)=1$,
la contribution de $F\times\{\sigma\}$ à
$\int_{Q}^\partial
\langle \dc t\wedge u_1\alpha,\mu\rangle$ est égale à 
\[
- \int_{F\times\{\sigma\}}
      \langle u_1 \alpha\wedge\dc t,\mu_F\cdot \abs e\rangle
 = - \int_F \langle \alpha,\mu_F\rangle,\]
 si bien que la contribution totale de
 $P\times \{\sigma\}$ à l'intégrale
 étudiée est égale à $-\int_P\langle
 \alpha,\mu\rangle$ soit encore à 
 $-\int_{V_{E,0}} \omega $. 
 
La contribution de $P\times	\{\tau\}$ est quant à elle 
nulle car
la fonction $u_1$ est nulle au voisinage de $\tau$. 
Enfin, la contribution de $\partial (P)
\times[\sigma,\tau]$ elle est 
également nulle : elle est en effet égale par construction à une somme
d'intégrales de la forme $\int_{F\times [\sigma,\tau]}u_1\eta$ où
$F$ est une $(n-2)$-cellule orientée de $\mathscr C$ et $\eta$ une $(n-1)$-forme
différentielle sur $F\times  [\sigma,\tau]$ qui provient en réalité
d'une $(n-1)$-forme différentielle sur $F$ (l'élément 
différentiel $\dc t$ ayant été contracté) et est dès lors nulle. 

Il vient 
\[\int^\partial_{V_{\geq s}}  \dc(\log\abs f) \wedge \omega_1
=-\int_{V_{E,0}}\omega,\]
ce qui conclut la démonstration du théorème~\ref{theo.poincare-lelong}.

\section{Support des courants positifs fermés}

\subsection{}
Soit $X$ un $n$-espace analytique, 
considéré comme espace tropical de niveau~$n$.
Le support d'une $(n,n)$-forme~$\alpha$ sur~$X$ est contenu dans une
partie squelettique et possède une densité par
rapport aux mesures de Lebesgue $n$-dimensionnelles.
En particulier, la mesure de Radon associée
ne charge pas des parties squelettiques de dimension de Hausdorff~$<n$.

Cette section vise à étendre cette propriété
à certains courant de la forme $T\wedge \alpha$,
où $T$ est un courant positif fermé de bidimension~$(p,p)$
et $\alpha$ est une forme lisse de bidegré~$(p,p)$.

\begin{prop}\label{rema.courants-chgt-base}
Soit $X$ un $n$-espace analytique, sans bord.
Soit $K$ une extension valuée de~$k$, universellement multiplicative;
soit $p\colon X_K\to X$ le morphisme de changement de base
et soit $\sigma\colon X\to X_K$ sa section de Shilov.

Soit $u^1,\dots,u^p$ des fonctions dpsh sur~$X$,
soit $T=\ddc(u^1)\wedge\dots\wedge \ddc(u^p)$
et soit $T_K=\ddc(u^1_K)\wedge\dots\wedge\ddc(u^p_K)$.
Soit $\alpha$ une forme lisse sur~$X$ de bidegré~$(u,v)$,
où $\max(u,v)\geq n-p$.
Le support du courant $T_K \wedge p^*\alpha$
est contenu dans~$\sigma(X)$.
\end{prop}
La proposition s'applique par exemple lorsque $K$ est le corps $\hr{\eta_r}$
où $r=(r_1,\dots,r_d)$ est une suite de nombres réels strictement positifs
et $\eta_r$ est le point de Gauss correspondant de~$\gm^d$.
Il s'applique également lorsque $k$ est algébriquement clos et $K$ arbitraire.

Rappelons aussi la proposition~\ref{prop.bt-projection-p} 
selon laquelle $\pi_*(T_K)=T$.

%
%
%
%

\begin{proof}
Le résultat est local sur~$X$; par multilinéarité
on peut donc se ramener au cas où les~$u_i$ sont psh.
Considérons les suites $(u^1,\dots,u^p )$ de fonctions psh,
définies sur un ouvert de~$X$,
pour lesquelles le support de~$T_K\wedge\pi^*\alpha$
est contenu dans~$\sigma(X)$.
Elles forment un sous-faisceau en $\R_+$-modules
du faisceau~$(\CPSH_X)^p$
qui est stable par limite uniforme, d'après la 
continuité du produit des courants à la Bedford--Taylor.
Il suffit donc de prouver qu'il contient toute
suite $(u^1,\dots,u^p)$ de fonctions lisses psh 
définies sur un ouvert de~$X$.
La propriété est locale sur~$X$, ce qui permet de supposer
que $\alpha$ et les~$u_i$ sont tropicalisées par un moment commun~$f$.
Dans ce cas, $T_K\wedge\pi^*\alpha$ est le courant associé
à la forme lisse $\pi^*(\ddc(u_1)\wedge\dots\wedge \ddc(u_p)\wedge\alpha)$.
Son support est donc contenu dans~$\Sigma_{f_K}^{(n)}$,
lequel est égal à~$\sigma(\Sigma_f^{(n)})$ 
en vertu du corollaire~\ref{coro.Sigma-chgt-base}.
En particulier, il est contenu dans~$\sigma(X)$, ce qu'il fallait
démontrer.
\end{proof}

\begin{defi}
Soit $X$ un espace G-tropical
et soit~$A$ une partie fermée de~$X$. Soit $p$ un nombre réel~$\geq0$.
On dit que $A$ est tropicalement de mesure de Hausdorff
$p$-dimensionnelle nulle, ou tropicalement $\mathscr H^p$-négligeable,
si pour toute carte G-tropicale $f\colon V\to \R^m$ 
sur un domaine compact~$V$ de~$X$, l'ensemble $f(A\cap V)$ de $\R^m$
est de mesure de Hausdorff $p$-dimensionnelle nulle.
\end{defi}
C'est une propriété G-locale sur~$X$.

\begin{exem}\label{exem.fan-neg}
Lorsque $X$ est un espace analytique topologiquement séparé,
toute partie faiblement analytique de dimension~$<p$ 
est tropicalement $\mathscr H^p$-négligeable,
puisque la tropicalisation d'un espace analytique compact
de dimension~$<p$ est une partie paralinéaire 
de dimension~$<p$.
\end{exem}

\begin{theo}\label{theo.negligeable-courant}
On suppose que la valeur absolue de~$k$ est non triviale.
Soit~$X$ un $n$-espace analytique sans bord
et soit~$T$ un courant positif fermé de bidimension~$(p,p)$ sur~$X$.
Soit $A$ une partie fermée de~$X$ qui est tropicalement
$\mathscr H^p$-négligeable.
Pour toute forme~$\alpha\in\mathscr A^{p,p}(X)$,
l'ensemble~$A$ est négligeable pour la mesure $T\wedge\alpha$ sur~$X$.
\end{theo}
\begin{proof}
L'assertion à prouver est locale sur~$X$. 
On peut ainsi supposer que $X$ est paracompact 
et que la forme~$\alpha$ est tropicale.
Soit $x\in X$. Soit $V$ un voisinage analytique compact de~$X$
et soit $f\colon V\to \gm^n$ un moment tel que $f_\trop(x)\not\in f_\trop(\partial(V))$; il en existe d'après le lemme~\ref{lemm.moment-bord}.
Quitte à concaténer~$f$ avec un moment tropicalisant~$\alpha$,
on peut supposer que $f$ tropicalise~$\alpha|_V$.
Soit $\omega$ une forme tropicale au voisinage de~$f_\trop(V)$
telle que $\alpha|_V=f^*\omega$. En la multipliant
par une fonction-plateau convenable, on se ramène au cas
où $\omega$ est une forme de type~$(p,p)$ sur~$\R^n$, à support compact.

Posons $\Omega=\R^n\setminus f_\trop(\partial(V))$; c'est
un ouvert de $\R^n$ car $\partial(V)$ est compact.
Posons alors
 $U=f_\trop^{-1}(\Omega)=V\setminus f_\trop^{-1}(f_\trop(\partial(V)))$;
c'est un ouvert de~$V$ qui est disjoint de~$\partial(V)$; 
il est donc contenu dans~$\Int(V)$ et est donc ouvert dans~$X$.
Il contient aussi~$x$.
Par construction, l'application $f_\trop|_U \colon U\to \Omega$
 est propre, puisqu'elle se déduit de~$f_\trop\colon V\to \R^n$ 
par changement de base.
On définit alors un courant tropical de type~$(p,p)$, $T_\Omega$, sur~$\Omega$,
par la formule $\beta\mapsto \langle T|_U, f^*\beta\rangle$,
pour $\beta\in\mathscr A^{p,p}_\cpct(\Omega)$.
Ce courant est positif et fermé, car il en est de même de~$T$.
L'hypothèse faite sur~$A$ entraîne que $f_\trop(A\cap U)$ est 
de mesure de Hausdorff $p$-dimensionnelle nulle dans~$\R^n$.
D'après~\cite[corollary~4.1]{lagerberg2012}, 
la mesure $T_\Omega\wedge\omega$ sur~$\Omega$ ne charge pas~$f_\trop(A\cap U)$.
Cela entraîne que la mesure $T\wedge\alpha$ sur~$U$
ne charge pas $f_\trop ^{-1}(f_\trop (A\cap U))$,
donc ne charge pas $A\cap U$.

Précisons cet argument.
L'ouvert~$\Omega$ de~$\R^n$ est réunion dénombrable de parties compactes;
comme $f_\trop\colon U\to\Omega$ est propre, l'espace~$U$ est réunion dénombrable
de parties compactes et il en est de même de son sous-espace fermé~$A\cap U$.
Il suffit de démontrer que pour toute partie compacte~$B$ de~$A\cap U$,
la mesure~$T\wedge\alpha$ ne charge pas~$B$.
Comme $f_\trop(B)$ est une partie compacte de~$f_\trop(A\cap U)$,
elle est de mesure de Hausdorff $p$-dimensionnelle nulle dans~$\R^n$.
Considérons une suite $(u_n)$ de fonctions continues
et positives sur~$\Omega$, à supports compact, à valeurs dans~$[0;1]$,
qui converge simplement vers la fonction indicatrice de~$f_\trop(B)$.
On peut, par exemple, prendre la fonction $u_n=\sup(1-  d(\cdot,f_\trop(B)/c)^{1/n},0)$, où $c$ est assez petit pour que le $c$-voisinage de $f_\trop(B)$
soit contenu dans~$\Omega$.
Par densité des fonctions lisses dans les fonctions continues,
il existe une suite de fonctions lisse à support compacts 
contenus dans~$\Omega$ telle que $\abs{v_n-u_n} \leq 1/n$. 
Alors, la suite bornée $(v_n)$ converge vers la fonction 
indicatrice de~$f_\trop(B)$, donc
l'intégrale $\langle T_\Omega\wedge\omega, v_n\rangle $ tend vers~$0$.
Par définition, on a donc $\langle T\wedge\alpha,f^*v_n\rangle \to 0$.
Hors, $(f^*v_n)$ est est une suite bornée qui converge vers
la fonction indicatrice de $f_\trop^{-1}(f_\trop(B))$;
par suite, $f^{-1}_\trop (f_\trop(B))$ est négligeable pour
la mesure~$T\wedge\alpha$, et c'est a fortiori le cas de~$B$.
\end{proof}

\begin{coro}\label{coro.support-prod-ddc}
Soit $X$ un $n$-espace analytique sans bord,
soit $p$ un entier naturel
et soit $(u_1,\dots,u_p)$ une famille de fonctions dpsh sur~$X$.
Soit $A$ une partie fermée de~$X$
qui est tropicalement~$\mathscr H^p$-négligeable.
Pour toute forme $\alpha\in\mathscr A^{n-p,n-p}(X)$, 
l'ensemble~$A$ est négligeable pour la mesure $
\ddc u_1 \wedge \dots \wedge \ddc u_p \wedge \alpha$ sur~$X$.
\end{coro}
\begin{proof}
Lorsque la valuation de~$k$ est non triviale, ce résultat est un cas
particulier du théorème~\ref{theo.negligeable-courant};
ce qui suit n'est donc intéressant que si la valuation de~$k$
est triviale, ce qu'on suppose désormais.

L'assertion est locale sur~$X$.
Par multilinéarité, on  peut supposer
que les~$u_i$ sont des fonctions psh.
Notons $T$ le courant $\ddc u_1\wedge \dots \wedge \ddc u_p$;
il est positif et fermé.

Soit $r$ un nombre réel tel que $0<r<1$ et soit $k_r$
l'extension complète~$k\lpar T\rpar$ de~$k$
munie de la valeur absolue $T$-adique telle que $\abs T=r$.
Soit $X_r$ l'espace $k_r$-analytique déduit de~$X$
par changement de base, soit $\pi\colon X_r\to X$ sa projection
et soit $\sigma\colon X\to X_r$ sa section de Shilov.

L'image~$\sigma(A)$ est une partie fermée de~$X_r$ (car $X$ est topologiquement
séparé).
Démontrons qu'elle est tropicalement $\mathscr H^p $-négligeable.
On peut raisonner G-localement et supposer que $X$ est affinoïde.
Soit $W$ un domaine analytique compact de~$X_L$ et soit $f$
un moment sur~$W$. En utilisant le lemme~\ref{lemm.chgt-base-kr}
et en utilisant un recouvrement fini de~$\sigma^{-1}(W)$, 
on se ramène au cas où il existe un moment~$g$ au voisinage de~$A$
tel que $f_\trop \circ \sigma = g_\trop  + e\log(r)$, où $e\in\Z^m$.
Par hypothèse, $g_\trop(A)$ est $\mathscr H^p$-négligeable;
il en est donc de même de $f_\trop(\sigma(A))$.

Pour tout~$j$, posons $u_{j,r}=\pi^* u_j$ et notons~$T_r$
le courant $\ddc u_{1,r}\wedge\dots \wedge \ddc u_{p,r}$ sur~$X_r$;
il est positif fermé.
Posons aussi $\alpha_r=\pi^*\alpha$.

D'après la remarque~\ref{rema.courants-chgt-base}, 
on a l'égalité de courants~$\pi_*(T_r)=T$ et 
et l'égalité de mesures $\pi_*(T_r\wedge\alpha_r)=T\wedge \alpha$.
Toujours d'après cette remarque, le support de $T_r\wedge\alpha_r$
est contenu dans l'image~$\sigma(X)$ de la section de Shilov.
D'après le théorème~\ref{theo.negligeable-courant}, cette mesure
sur~$X_r$ ne charge pas $\sigma(A)$.
Elle ne charge donc pas $\pi^{-1}(A)$.
Par suite, $T\wedge\alpha$ ne charge pas~$A$, ce qu'il fallait démontrer.
\end{proof}
%
%
%
%

\begin{lemm}\label{lemm.chgt-base-kr}
Soit $X=\mathscr M(\mathscr A)$ un espace $k$-affinoïde,
soit $W$ un domaine analytique compact de~$X_L$
et soit $f\in\mathscr A_L^\times$ une fonction analytique sur~$W$, inversible.
Il existe une suite finie $(K_1,\dots,M_m)$ de parties compactes
de~$\sigma^{-1}(W)$ et pour tout~$j$, une fonction analytique~$g_j$ 
sur~$X$ et un entier~$e_j$,
telle que $\abs{f}\circ\sigma=\abs{g_j} r^{e_j}$ sur~$K_j$.
\end{lemm}
\begin{proof}
En utilisant le théorème de Gerritzen--Grauert, on se ramène
au cas où $W$ est un domaine rationnel de~$X_L$. 
Il existe donc deux fonctions analytiques $\phi$ et~$\psi$ sur~$X_L$
telles que $\psi$ ne s'annule pas sur~$W$ 
et telles que $\abs f = \abs \phi/\abs\psi$ sur~$W$.
On est ainsi ramené au cas où $f$ est la restriction d'une fonction
analytique~$\phi$ sur~$X_L$ puis, par approximation,
au cas où $\phi\in\mathscr A[T,T^{-1}]$.
Écrivons $\phi=\sum_j f_j T^j$,
où $f_j\in\mathscr A$ pour tout~$j$. Pour tout $x\in X$,
on a $\abs \phi \circ \sigma (x) = \max \left( \abs{f_j}(x) r^j \right)$.
On peut ainsi prendre pour partie compacte~$K_j$
l'ensemble des points~$x\in \sigma^{-1}(W)$ 
tels que $\abs \phi \circ \sigma (x) = \abs{f_j(x)} r^j $.
\end{proof}

\section{Négligeabilité des ensembles polaires}

\begin{prop}\label{prop.vois-pseudoconv}
Soit $X$ un $n$-espace analytique sans bord 
sur un corps complet non archimédien~$k$.
On suppose que $X$ est bon.
Tout point de~$X$ possède une base de voisinages ouverts~$U$
sur lesquels il existe une fonction psh lisse~$\rho\colon U\to\R$
telle que $\rho<0$ sur~$U$ et $\{\rho\leq 0\}$ est compact.
\end{prop}
\begin{proof}
Comme tout ouvert de~$X$ est bon, il suffit de démontrer
que tout point~$x$ de~$X$ possède un voisinage ouvert comme dans l'énoncé.
Comme $X$ est bon, il existe un domaine affinoïde~$V$ de~$X$ 
qui est un voisinage de~$x$.
Soit $W$ un voisinage affinoïde de~$x$ contenu dans $\Int(V/X)$
qui est un domaine rationnel: il existe des fonctions holomorphes
$f_1,\dots,f_n,g$ sur~$V$ sans zéro commun 
et des nombres réels strictement positifs
$\lambda_1,\dots,\lambda_n$ strictement 
tel que $W$  soit défini par les conditions $\abs{f_i}\leq \lambda_i \abs g$
(\cite{berkovich1990}, \S2.2).
Soit $U$ l'ouvert d'inversibilité de~$V$.
Pour tout nombre réel~$\mu\geq 1$,
$W_{\mu}=\{ \abs{f_i}\leq \lambda_i \mu \abs g \}$
est un voisinage rationnel de~$x$ dans~$V$ qui est contenu
dans~$U$, car les~$f_i$ et~$g$ sont sans zéro commun.

Pour tout $i$, posons $u_i=\log (\abs{f_i/g}/\lambda_i)$;
c'est une fonction psh sur~$U$,
qui vaut $-\infty$ sur $\{f_i=0\}\cap U$
est est paralinéaire en dehors.
La fonction  $u=\max(u_1,\dots,u_n)$ sur~$U$ est paralinéaire
et psh,
et pour tout $t\geq 0$,
l'ensemble $\{u\leq t\}$ est un voisinage de~$x$ dans~$U$.
Nous allons en construire une approximation psh lisse.

Soit $\theta$ une fonction lisse sur~$\R$,
positive, de support contenu dans~$[-1;1]$
et d'intégrale~$1$. 
Pour tout $\eps>0$ et tout~$s\in\R^m$, posons
\[ M_\eps (s_1,\dots,s_m) = \int_{\R^m} \max(s_1+t_1,\dots,s_m+t_m)
\prod_{i=1}^m \theta(t_i/\eps) \, dt_1\cdot dt_m.\]
D'après le lemme~5.18 de~\cite{demailly2009}, 
la fonction~$M_\eps$ sur~$\R^m$
est lisse, convexe, croissante en chaque variable,
vérifie 
\[ \max(s_1,\dots,s_m) \leq M_\eps(s_1,\dots,s_m) \leq \max(s_1,\dots,s_m)+\eps.\]
En outre, si $s'$ est la famille déduite de~$s$ en enlevant
les coordonnées~$x_i$ telles que $s_i\leq \max(s_1,\dots,s_m)-2\eps$,
on a $M_\eps(s)=M_\eps(s')$.
Par suite, la fonction $u_\eps$ sur~$U$
définie par $u_\eps= M_\eps(u_1,\dots,u_n)$ est lisse et psh,
et vérifie $u\leq u_\eps\leq u+\eps$.

Prenons $\eps=1$ et posons $\tilde u=u_1-u_1(x)-1$.
La fonction~$\tilde u$ sur~$U$ est psh et lisse ;
on a $\tilde u(x)=-1<0$.
L'ensemble de~$U$ défini par $\tilde u\leq 0$
est contenu dans l'ensemble fermé $\{u \leq u_1(x)+1\}$ de~$U$.
Or, ce dernier ensemble est aussi le domaine rationnel de~$V$
défini par $\abs{f_i}\leq e^{u_1(x)+1}\lambda_i \abs g$,
car les~$f_i$ et~$g$ sont sans zéro commun ; il est donc compact.
\end{proof}

\begin{theo}\label{theo.psh-locint}
Soit~$X$ un $n$-espace $k$-analytique sans bord.
Soit~$u_1,\dots,u_m$ des fonctions psh continues sur~$X$
qui sont localement approchables,
soit~$V$ une fonction psh localement approchable sur~$X$,
soit $\alpha$ une forme lisse sur~$X$.
Si $V$ est localement intégrable sur~$X$, est elle est localement intégrable
pour la mesure $\ddc u_1\wedge\dots\ddc u_m\wedge\alpha$ sur~$X$.
\end{theo}

\begin{proof}
D'après la proposition~\ref{prop.vois-pseudoconv},
l'espace tropical associé à~$X$ est justiciable 
du corollaire~\ref{coro.local-int}.
\end{proof}
\begin{coro}
Soit $X$ un $n$-espace $k$-analytique sans bord.
Soit $u_1,\dots,u_m$ des fonctions psh continues sur~$X$
qui sont localement approchables,
soit $\alpha$ une forme lisse de type~$(n-m,n-m)$ sur~$X$
et soit $V$ une fonction psh localement approchable sur~$X$.
Si $V$ est localement intégrable sur~$X$, 
la mesure $\ddc u_1\wedge\dots\ddc u_m\wedge\alpha$ sur~$X$
ne charge pas le lieu polaire de~$V$.
\end{coro}
\begin{proof}
D'après la proposition~\ref{prop.vois-pseudoconv},
l'espace tropical associé à~$X$ est justiciable 
du corollaire~\ref{coro.local-polaire}.
\end{proof}
\begin{coro}\label{coro.zariski-bt-neg}
Soit $X$ un $n$-espace $k$-analytique sans bord.
Soit $u_1,\dots,u_m$ des fonctions psh continues sur~$X$
qui sont localement approchables,
soit $\alpha$ une forme lisse de type~$(n-m,n-m)$ sur~$X$.
La mesure $\ddc u_1\wedge\dots \wedge \ddc u_m\wedge\alpha$ sur~$X$
ne charge aucun fermé de Zariski de dimension~$<n$.
\end{coro}
\begin{proof}
Soit $Y$ un fermé de Zariski de~$X$ tel que $\dim(Y)<\dim(X)$.
L'assertion est locale au voisinage de tout point~$y$ de~$Y$. 
Soit $W$ un voisinage affinoïde de~$y$.
La trace $Y\cap W$ ne contient aucune composante irréductible de~$W$
de dimension~$n$.
D'après le lemme d'évitement des idéaux premiers, 
il existe une fonction analytique~$f$ sur~$W$ telle que $Y\subset V(f)$ 
et telle que $f$ ne s'annule identiquement 
sur aucune composante irréductible de~$W$ de dimension~$n$.
Quitte à remplacer~$X$ par l'intérieur de~$W$,
on peut supposer qu'il existe une fonction analytique~$f$
sur~$X$ qui ne s'annule identiquement sur aucune composante
irréductible de dimension~$n$ et~$X$ et telle que $Y\subset V(f)$.

La fonction $\log \abs f$ est alors continue sur toute
partie squelettique de~$X$, donc est localement intégrable sur~$X$.
Par ailleurs, c'est la limite décroissante des fonctions 
$\log\max(\abs f, t)$, pour $t\to 0$. Comme ces fonctions sont psh
et continues sur~$X$, il en résulte que $\log\abs f$ est psh approchable.
Comme son lieu polaire contient~$Y$, le corollaire en découle.
\end{proof}

\begin{coro}
Reprenons les hypothèses du corollaire précédent.
Soit $Y$ un $n$-espace $k$-analytique
et soit $f\colon Y\to X$ un morphisme topologiquement
compact qui est fini
au-dessus d'un ouvert de Zariski  dense.
Pour toute fonction $\delta$ de~$X$ dans~$\R$ qui prolonge
la fonction~$\degcal_f$, 
on a l'égalité de courants de bidegré~$(m,m)$
\[ f_* (\ddc (f^* u_1)\wedge \dots \wedge \ddc(f^* u_m))
   = \delta \cdot \ddc u_1\wedge\dots \wedge \ddc u_m \]
sur~$X$.
\end{coro}
\begin{proof}
Dans la formule précédente,
le membre de gauche est l'image directe d'un courant positif sur~$Y$,
donc est un courant positif sur~$X$.
Quant à lui, le membre de doite est le produit d'un courant positif
$ \ddc u_1\wedge\dots \wedge \ddc u_m$
par la fonction~$\delta$;
nous allons prouver qu'il s'agit d'un courant
en démontrant que pour toute forme lisse~$\alpha$ à support compact,
de type~$(n-m,n-m)$,
la fonction~$\delta$
est localement intégrable contre
la mesure signée
$\mu_\alpha = \ddc u_1\wedge\dots\wedge \ddc u_m\wedge [\alpha]$,
et que l'on a
\[ \langle \ddc (f^* u_1)\wedge \dots \wedge \ddc(f^* u_m), f^*\alpha\rangle
   = \int_X \delta \cdot \ddc u_1\wedge\dots \wedge \ddc u_m \wedge [\alpha].
\]

Par partition de l'unité, on peut supposer que le
support de~$\alpha$ est contenu
dans un ouvert de~$X$ sur lequel les~$u_i$ sont limite
uniforme de suites de fonctions lisses psh.
Quitte à remplacer~$X$ par un tel ouvert, on suppose
donc que les~$u_i$ sont limite uniforme de suites
de fonctions lisses psh sur~$X$.

La forme~$\alpha$ est différence de deux formes
positives à support compact, 
si bien qu'on peut également supposer que $\alpha$ est positive.
La mesure~$\mu_\alpha$ est alors positive.

La réunion~$Z_1$ des intersections deux à deux des composantes irréductibles
de~$X$ est de dimension~$<n$.
Soit $Z_2$ un fermé de Zariski de~$X$ de dimension~$<n$,
de complémentaire~$U$, tel que $f$ soit fini au-dessus de~$U$.
Posons $Z=Z_1\cup Z_2$ et soit~$U$ son complémentaire;
on a $\dim(Z)<n$ et $f$ est fini au-dessus de~$U$.

La formule voulue est vraie lorsque les~$u_i$ sont lisses.
En effet, 
$\mu_\alpha$ est alors la mesure associée à la $(n,n)$-forme
positive lisse
$ \omega=\ddc u_1\wedge\dots\wedge \ddc u_m\wedge\alpha$;
elle possède un support fort~$S$ qui est contenu dans~$X^\gen$. 
Alors,
\[ \int_X \delta \mu_\alpha = \int_X \degcal_f \ddc u_1\wedge
\dots \wedge \ddc u_m\wedge \alpha. \]
Soit $U$ un voisinage ouvert de~$S$ dont chaque point
n'appartient qu'à une seule composante irréductible de~$X$
et soit $(U_i)$ la famille de ses composantes connexes.
On a donc 
\[ \int_X \delta \mu_\alpha = \int_U \degcal_f \ddc u_1\wedge
\dots \wedge \ddc u_m\wedge \alpha
=\sum_i \int_{U_i} \degcal_f \ddc u_1\wedge\dots\wedge\ddc u_m\wedge\alpha.\]
Pour tout~$i$, posons $V_i=f^{-1}(U_i)$; posons aussi $V=f^{-1}(U)$;
on a
\[ \int_{U_i}\degcal_f \ddc u_1\wedge\dots\wedge\ddc u_m\wedge\alpha
= \int_{V_i} \ddc f^*u_1\wedge\dots\wedge\ddc f^*u_m\wedge f^*\alpha, \]
de sorte
que 
\[ \int_X \delta \mu_\alpha = \int_V \ddc f^*u_1\wedge\dots\wedge\ddc f^*u_m\wedge f^*\alpha, \]
La $(n,n)$-forme lisse
$\ddc f^*u_1\wedge\dots\wedge\ddc f^*u_m\wedge f^*\alpha = f^*\omega$
sur~$Y$ possède un support fort~$T$.
Si une $n$-cellule~$C$ de~$T$ est contractée par~$f$,
alors $f^*\omega$ est identiquement nulle sur~$C$.
On peut donc supposer qu'aucune $n$-cellule de~$T$ n'est contractée
par~$f$, si bien que $T\subset V$ et
\[ \int_V \ddc f^*u_1\wedge\dots\wedge\ddc f^*u_m\wedge f^*\alpha
 = \int_X \ddc f^*u_1\wedge\dots\wedge\ddc f^*u_m\wedge f^*\alpha.\]
Cela prouve l'assertion voulue quand les~$u_i$ sont lisses.

Dans le cas général, on considère des suites~$(u_{i,p})_p$ de fonctions
lisses psh sur~$X$ qui convergent uniformément vers~$X$.
Pour tout~$p$, on a l'égalité
\[ \langle f_* (\ddc (f^* u_{1,p})\wedge \dots \wedge \ddc(f^* u_{m,p})),
\alpha\rangle 
   = \int_X \delta \cdot \ddc u_{1,p}\wedge\dots \wedge \ddc u_{m,p} \wedge [\alpha].\]
En vertu du la proposition~\ref{prop.BT-continu},
le membre de gauche converge vers
\[ \langle f_* (\ddc (f^* u_{1})\wedge \dots \wedge \ddc(f^* u_{m})),
\alpha\rangle. \]
Le membre de droite est de la forme
$\int_X \delta \mu_p$, où $(\mu_p)$ est une suite de mesures
positives sur~$X$ qui converge vers la mesure~$\mu_\alpha$;
les supports de ces mesures sont compacts,
contenus dans celui de~$\alpha$. 
La fonction~$\delta$ est localement constante, donc continue, sur~$U$.
D'après le corollaire~\ref{coro.zariski-bt-neg}, le complémentaire
de~$U$ est négligeable pour la mesure~$\mu_\alpha$.
On a donc (proposition~22 de \cite{INT-1-4}, chap.~IV, §5, \no12)
$\int_X \delta\mu_p\to\int_X\delta\mu_\alpha$;
cela conclut la démonstration.
\end{proof}

\part{Fibrés métrisés et courbures}

\chapter{Fibrés vectoriels métrisés}
\label{sec.metriques}

\section{Métriques}

\subsection{Fibrés vectoriels}

Un fibré vectoriel de rang~$n$ sur un espace $k$-analytique~$X$ 
est un faisceau en modules localement libres de rang~$n$
constant sur le site 
annelé $X_{\rm G}$ (à savoir~$X$ muni
de sa G-topologie et de son faisceau structural).

 Il définit par restriction 
à l'espace topologique $X$ un faisceau en 
$\mathscr O_X$-modules, dont on ne sait pas s'il
est en général localement libre de rang~$n$
C'est toutefois toujours le cas lorsque $X$ est bon, et notamment 
lorsque $X$ est sans bord.
Dans cette situation, cette construction fournit
une équivalence entre la catégorie des 
$\mathscr O_{X_{\rm G}}$-modules 
localement libres de rang fini et celle des
$\mathscr O_X$-modules localement libres de rang
fini (\cite{berkovich1993}, prop.  1.3.4).

Soit $X$ un espace analytique et soit $E$ un fibré vectoriel sur~$X$.
Notons $\mathbf V(E)$ l'espace total de~$E$: c'est un espace
analytique muni d'un morphisme $\pi$ vers~$X$  dont les
sections au-dessus d'un domaine analytique de~$X$ s'identifient
canoniquement aux sections de~$E$ sur ce domaine.
Si $E= \mathscr O_X^n$, 
alors $\mathbf V(E)=\A^n_X$ est l'espace affine (analytique)
au-dessus de~$X$; dans le cas général, on le construit par recollement.

\begin{defi}[Métriques]
Soit $X$ un espace analytique et soit $E$ un fibré vectoriel sur~$X$.
Une \emph{semi-métrique}\index{semi-métrique} sur~$E$
est une application continue $\norm\cdot\colon \mathbf V(E)\ra\R_+$
qui est une norme dans chaque fibre, au sens suivant:
pour tout point $x\in X$, toute extension valuée complète~$L$
de~$\mathscr H(x)$, l'application de $E(x)\otimes L$
dans~$\R$ donnée par la composition
\[ E(x)\otimes_{\mathscr H(x)} L \hra  \mathbf V(E)_L \ra \mathbf V(E)\xrightarrow{\norm\cdot }  \R_+ \]
est une semi-norme ultramétrique
sur le $L$-espace vectoriel de dimension finie $E(x)\otimes L$.

On dit que cette semi-métrique est une métrique\index{métrique}
si ces semi-normes sont des normes et qu'elle est continue
si l'application $\norm\cdot\colon \mathbf V(E)\to\R_+$ ci-dessus
est continue.
\end{defi}

On notera souvent $\bar E$ la donnée d'un fibré vectoriel~$E$ sur~$X$
et d'une semi-métrique sur ce fibré.
Inversement, si $F$ est un fibré semi-métrisé sur~$X$, 
on notera~$\underline F$ le fibré vectoriel sous-jacent.

\begin{rema}
L'exigence d'une telle condition après produit tensoriel par
toute extension valuée complète~$L$ de~$\hr x$ peut sembler superflue.
Que nenni! Prenons par exemple le fibré de rang~$1$ trivial~$E$ sur
l'espace~$\mathscr M(k)$ d'un corps trivialement valué~$k$. 
Toute fonction continue de $\mathbf V(E)$ dans~$\R_+$ qui applique
tout $k$-point non nul sur~$1$ et l'origine sur~$0$
induit une norme sur~$E(k)$. Si c'est le cas de la fonction $\abs T$,
c'est également le cas de la fonction $\abs T^2$ qui n'induit
pas de norme sur $E(x)\otimes L$ dès que $L$ n'est plus trivialement valuée.
\end{rema}

\begin{exem}\label{exem.metrique-rang-1}
Soit $X$ un espace analytique, soit $L$ un fibré en droites sur~$X$.

Soit $\eps$ une section globale \emph{inversible} de~$L$.
Pour toute semi-métrique~$\norm\cdot$ sur~$L$,
on dispose d'une fonction $\norm \eps$ sur~$X$.
Inversement, pour toute fonction $u\colon X\to\R_+$,
il existe une unique semi-métrique sur~$L$
telle que $u= \norm\eps$. En effet, $\eps$
identifie $\mathbf V(L)$ à $\A^1_X$ 
et la semi-métrique est alors donnée par $v\mapsto \abs{T(v)} u(p(v))$,
où $p\colon \A^1_X\to X$ est la projection naturelle.
Les métriques correspondent ainsi aux fonctions à valeurs
strictement positives, et les métriques continues aux fonctions
continues à valeurs strictement positives.

Soit $(U_i)$ un G-recouvrement de~$X$ trivialisant~$L$;
pour tout~$i$, soit $s_i$ une section inversible de~$L$ sur~$U_i$;
pour tout couple $(i,j)$, soit $f_{ij}$ la fonction inversible
$s_i/s_j$  sur~$U_i\cap U_j$.
Soit $\norm\cdot$ une semi-métrique sur~$L$.
Pour tout~$i$, soit $u_i$ la fonction~$\norm{s_i}$ sur~$U_i$;
on a $u_i= \abs{s_i/s_j} u_j$ pour tout couple~$(i,j)$.
Inversement, si $(u_i)$ est une famille de fonctions sur les domaines~$U_i$
vérifiant cette relation de cocycle, il existe une unique
semi-métrique~$\norm\cdot$ sur~$L$ telle que $\norm{s_i}=u_i$ pour tout~$i$.

Cette semi-métrique est une métrique si les fonctions~$u_i$
sont à valeurs strictement positives, et est une métrique continue
si ces fonctions sont de plus continues.
Dans ce cas, les fonctions continues $\log \norm{s_i}^{-1}$ sur les~$U_i$
ne se recollent pas a priori mais c'est le cas de leurs
images dans le faisceau quotient $\mathscr C_X/\log(\abs{\mathscr O_X^\times})$.
\end{exem}

\begin{rema}\label{rema.metrique.G-local}
Soit $X$ un espace analytique, soit $E$ un fibré vectoriel sur~$X$
muni d'une métrique continue~$\norm\cdot$.

Soit $Y$ un espace analytique, $f\colon Y\to X$ un morphisme
et $s$ une section globale de $f^*E$.
On déduit de~$\norm\cdot$ une application continue,
notée $\norm s$, de~$Y$ dans~$\R_+$:
c'est la composition
\[  Y \xrightarrow s \mathbf V(f^*E)\to \mathbf V(E)
  \xrightarrow{\norm\cdot}\R_+. \]
Cette application vérifie les propriétés suivantes : 
\begin{itemize}
\item 
Si $u\in\Gamma(Y,\mathscr O_Y)$ et
$s\in\Gamma(Y,f^*E)$,
alors $\norm{us} = \abs u\cdot \norm s$;
\item 
Si $s_1,s_2\in\Gamma(Y,f^*E)$,
alors $\norm{s_1+s_2} \leq\max(\norm {s_1}, \norm {s_2})$;
\item 
La fonction $\norm s$ est à valeurs positives 
et s'annule en un point $y$ de $Y$ si et seulement si $s(y)=0$;
\item
Soit $g\colon Z\to Y$ un morphisme d'espace analytiques,
on a $\norm{g^* s} = \norm s \circ g$.
\end{itemize}

Inversement, donnons-nous, pour tout espace analytique~$Y$,
tout morphisme~$f\colon Y\to X$, et toute section~$s$ de~$f^*E$,
une application $\norm s\colon Y\to \R_+$ vérifiant les propriétés
ci-dessus.  Elle est induite par une unique métrique continue sur le 
fibré~$E$: elle se retrouve en prenant pour espace~$Y$
l'espace total~$\mathbf V(E)$ et pour~$s$ la section tautologique.

Plus généralement, supposons donné un G-recouvrement~$(X_i)$ de~$X$
ainsi que de telles applications $\norm s$, pour tout triplet $(Y,f,s)$
comme ci-dessus lorsque $f(Y)$ est contenu dans l'un des~$X_i$.
Ces données fournissent d'abord une métrique sur le fibré $E|_{X_i}$,
pour tout~$i$; ces métriques coïncident sur les intersections
et se recollent en une métrique sur~$E$.

Notons $\mathscr C^0_X$ (resp.
$\mathscr C^0_{X_{\rm G}}$) le faisceau $U\mapsto \mathscr C^0(U,\R)$
sur l'espace topologique $X$ (resp. le site $X_{\rm G}$).
En se restreignant aux domaines analytiques de~$X$,
on voit en particulier qu'une métrique continue sur~$E$ donne lieu
à un morphisme de faisceaux $E\to \mathscr C^0_{X_{\rm G}}$.
\end{rema}

\begin{defi}[Métriques lisses]
On dit qu'une métrique sur un fibré vectoriel~$E$ est lisse
si l'application de~$\mathbf V(E)\setminus\{0\}$ dans~$\R_+^*$
déduite  par restriction au complémentaire  de la section
nulle est lisse.
\end{defi}

C'est une propriété locale sur~$X$.

\begin{prop}
Soit $L$ un fibré en droites
qui est localement trivial pour la topologie usuelle de~$X$.
Pour qu'une métrique continue sur~$L$ soit lisse, il faut et il suffit
que  pour toute section inversible~$s$ sur un ouvert~$U$
de~$X$, la fonction $\log\norm s$ soit lisse sur~$U$.
\end{prop}
Comme dans l'exemple, on déduit d'une métrique lisse sur~$L$
une section globale du faisceau $\mathscr A^0_X/\log(\abs{\mathscr O_X^\times})$.

\begin{proof}
La nécessité de la condition est évidente, car
$\norm s$ est la composition de la norme et de
la section $U\ra\mathbf V(E)$ déduite de~$s$,
l'image de cette section étant contenue dans le complémentaire
de la section nulle.
Inversement, soit $\norm\cdot$ une métrique continue sur~$X$
vérifiant la condition indiquée. Pour démontrer qu'elle est lisse,
on peut raisonner localement sur~$X$.
Par hypothèse, $X$ est recouvert par des ouverts~$U$
sur lesquels $L$ possède une section inversible~$s$.
La section~$s$ fournit un isomorphisme
de~$\mathbf V(E)|_U$ avec $\A^1\times U$; par cet
isomorphisme, la métrique est transformée en l'application
$(t,x)\mapsto \abs t \cdot \norm {s(x)}$; sa restriction
au complémentaire~$\gm\times U$ de la section
nulle est donc lisse. Par suite, la métrique donnée est lisse.
\end{proof}

\begin{prop}\label{prop.existence-metrique-lisse}
Soit $X$ un espace $k$-analytique paracompact.
Si $X$ est  localement holomorphiquement
séparé, tout fibré en droites sur~$X$ qui est localement
trivial pour la topologie usuelle de~$X$ possède une métrique lisse.
En particulier, si $X$ est bon (par exemple, sans bord),
tout fibré en droites sur~$X$ possède une métrique lisse.

De plus, toute métrique continue sur un tel fibré peut-être approchée
par une métrique lisse pour la topologie de la convergence compacte.
\end{prop}
\begin{proof}
Soit $L$ un fibré en droites sur~$X$.
Par hypothèse, 
il existe un recouvrement $(U_i)$ de~$X$ par des ouverts
sur lesquels $L$ est trivialisable;
pour tout~$i$, soit $s_i$ une section inversible de~$L|_{U_i}$.
Soit $(\lambda_i)$
une partition de l'unité lisse subordonnée au recouvrement ouvert~$(U_i)$;
il en existe, car $X$ est paracompact.
Il existe une unique métrique sur~$L$ pour laquelle,
\[ \log\norm s (x) = \sum_{i} \lambda_i(x) \log\abs{s/s_i}(x) \]
pour tout espace analytique~$Y$, tout morphisme $f\colon Y\to X$
et toute section inversible~$s$ de~$f^*L$ sur~$Y$,
la somme étant réduite aux indices~$i$ tels que $f(x)\in U_i$.

Lorsque $s$ est une section inversible sur un ouvert~$U$ de~$X$,
la fonction $\log\norm s$ est lisse sur~$U$,
si bien que la métrique ainsi définie sur~$L$ est lisse.

Soit $\norm\cdot$ une métrique lisse arbitraire sur~$L$
et soit $\norm\cdot'$ une métrique continue sur~$L$;
soit $K$ une partie compacte de~$X$ et soit~$\eps>0$.
Il existe une unique fonction continue~$f$ sur~$X$ telle
que $\norm\cdot' =\norm\cdot e^f$.
Soit $u$ une fonction lisse sur~$X$ telle que
$\abs{u(x)-f(x)}\leq\eps$ pour tout $x\in K$
(corollaire~\ref{coro.berk-sw}).
La métrique $\norm\cdot e^u$ sur~$L$ est lisse
et approche la métrique continue donnée~$\norm\cdot'$ comme voulu.
\end{proof}

\begin{rema}
En revanche, un fibré vectoriel de rang~$n>1$ sur un espace analytique~$X$
n'admet pas de métrique lisse: ce n'est déjà pas le cas lorsque
$X=\mathscr M(k)$. En effet, une métrique sur le fibré
vectoriel trivial~$\mathscr O_X^n$ est une application
$\norm\cdot$ de~$\A^n_k$ dans~$\R_+$ qui, pour toute extension
valuée~$L$ de~$k$, induit une norme ultramétrique sur~$L^n$.
Supposons, ce n'est pas restrictif,
que $k$ soit algébriquement clos et sphériquement complet.  
Dans ce cas,  on sait (\cite{perez-garcia-schikhof2010}, Theorem~2.3.25;
\cite{gruson1966}, prop.~2)
que l'espace vectoriel~$k^n$,
muni de la norme induite par~$\norm\cdot$, possède une base
$(e_1,\dots,e_n)$  qui est « orthogonale », au sens
où $\norm{\sum x_i e_i}=\max(\abs {x_i} \norm{e_i})$
pour tout $(x_1,\dots,x_n)\in k^n$.
Autrement dit, il existe des formes linéaires
linéairement indépendantes~$f_1,\dots,f_n$
sur~$k^n$ et des nombres réels strictement positifs
$a_1,\dots,a_n$ tels que $\norm{x}=\max(a_i \abs{f_i(x)})$
pour tout $x\in \A^n(k)$.
Comme $\A^n(k)$ est dense dans~$\A^n_k$, cette formule
vaut pour tout~$x\in\A^n_k$.

Cela entraîne que la norme~$\norm\cdot$ est~\pl au sens
du paragraphe suivant mais n'est pas lisse.
Supposons en effet par l'absurde
que~$\norm\cdot$ soit une métrique lisse.
La $(1,1)$-forme $\ddc \log\norm\cdot$ sur~$\A^{n}_k\setminus\{0\}$
descend en une $(1,1)$-forme lisse sur~$\P^{n-1}_k$.
Sa puissance $(n-1)$-ième est alors une $(n-1,n-1)$-forme lisse,
contredisant la proposition~\ref{MA-delta_0/1}
qui affirme que $(\ddc \log\norm\cdot)^{n-1}$
est une masse de Dirac.

C'est donc probablement par la géométrie de Finsler qu'il conviendra
d'étudier les formes de Chern des fibrés vectoriels en géométrie 
non archimédienne.
\end{rema}

\subsection{}
Le produit tensoriel~$L\otimes L'$ 
de deux fibrés en droites métrisés~$L$ et~$L'$
possède une unique métrique continue pour laquelle 
$\norm{s\otimes s'}=\norm s \cdot \norm {s'}$.

De même, l'inverse d'un fibré en droites métrisé~$L$
possède une unique métrique continue pour laquelle~$\norm {s^{-1}}=\norm s^{-1}$
quand $s$ est inversible.

Si $L$ et $L'$ sont des fibrés en droites
munis de métriques lisses, la métrique naturelle sur
$\underline L\otimes \underline L'$ est lisse, 
de même que la métrique naturelle de~$\underline L^{-1}$.

\section{Métriques paralinéaires}

Nous allons maintenant introduire une classe importante de fibrés métrisés, dont nous verrons plus bas
(\ref{s-metriques-formelles}) qu'elle contient notamment les fibrés métrisés issus de modèles formels.
Nous allons utiliser dans ce qui suit les différentes classes d'applications paralinéaires introduites en 
\ref{ss-classes-pl}.

\subsection{}
Soit $X$ un espace $k$-analytique, soit $E$
un fibré vectoriel sur $X$ et soit $(e_1,\ldots, e_n)$ une base de $E$. 

Pour toute  
famille
$(u_1,\ldots, u_n)$ de fonctions continues à valeurs strictement
positives sur $X$, il existe une unique métrique continue $\norm \cdot$ sur 
$E$ telle que \begin{equation}
\label{eq.base-orthogonale}
\norm{(\sum f_i \phi^*e_i)(y)}=\max  u_i(\phi (y)) \abs {f_i(y)}\end{equation}
 pour tout espace analytique $Y$, tout morphisme~$\phi\colon Y\to X$,
 toute famille $(f_i)$ de fonctions analytiques 
 sur $Y$ et tout $y\in Y$ ; nous noterons 
 $\max u_i \cdot \abs{e_i^\vee}$ cette métrique. \index{base orthogonale}

 Si $\norm \cdot $ est une métrique sur $E$, nous dirons que la base $(e_i)$
 est orthogonale pour $\norm \cdot$ s'il existe une famille $(u_i)$
 comme ci-dessus telle que $\norm \cdot =\max u_i\cdot \abs{e_i^\vee}$. La famille $u_i$ est alors uniquement
 déterminée : on a en effet nécessairement $u_i=\norm{e_i}$ pour tout $i$. Si $u_i$ est identiquement égale à $1$
 pour tout $i$ on dit que $(e_i)$ est orthonormée. 
\index{base orthonormée}

\subsection{}
Soit $(e_1,\dots,e_n)$ une base orthogonale de~$E$ 
et soit $(\eps_1,\dots,\eps_n)$ une famille de sections telle
que pour tout~$i$, l'inégalité $\norm{\eps_i}<\norm{e_i}$ 
soit vérifiée en tout point de~$X$.
On vérifie immédiatement que
la famille $(e_i+\eps_i)$  est une base orthogonale ; notons
que la fonction $\norm{e_i+\eps_i}$ est égale à $\norm{e_i}$ pour tout $i$. 

\begin{defi}\label{def.metrique-pl} 
 On dit qu'un fibré vectoriel métrisé $E$ est
 paralinéaire (\resp paralinéaire entier)
ou que la métrique de~$E$ est paralinéaire (\resp paralinéaire entière) 
s'il existe un G-recouvrement
 $(X_i)$ de $X$ et, pour tout $i$, 
 une base orthogonale $(e_{ij})$ telle que les fonctions $\log \norm{e_{ij}}$ 
 soient paralinéaires (resp. paralinéaires entières).
Si $X$ est strict, on dit que $E$
 est paralinéaire strict (\resp paralinéaire strictement entier) ou que sa métrique
 est paralinéaire stricte (\resp paralinéaire strictement entière) s'il existe un \Gstr -recouvrement $(X_i)$ de $X$
 et, pour tout $i$, 
 une base orthogonale $(e_{ij})$ telle que les fonctions $\log \norm{e_{ij}}$ 
 soient paralinéaires strictes (resp. paralinéaires strictement entières). 
\end{defi}

Notons que par définition, les fibrés paralinéaires stricts et strictement
entiers n'existent que sur les espaces stricts.  
Soit $E$ un fibré vectoriel muni d'une métrique paralinéaire;
pour toute section globale~$s$ de~$E$, la fonction
$\log \norm s$ est paralinéaire
sur l'ouvert de non-annulation de~$s$.
L'énoncé analogue vaut pour les métriques
paralinéaires entières, paralinéaires strictes, 
et paralinéaires strictement entières. 

\begin{rema}\label{rem-bases-orthogonales}
Soit $E$ un fibré vectoriel muni d'une métrique paralinéaire entière. Donnons-nous des familles
$(X_i)$ et $(e_{ij})$ comme dans la définition. Les fonctions $\log \norm{e_{ij}}$ sont paralinéaires entières. Pour $i$ fixé il y a 
un nombre fini de telles fonctions, et il existe donc un G-recouvrement $(X_{i\ell})$ 
de $X_i$ et, pour tout $\ell$, des fonctions analytiques inversibles
$g_{ij\ell}$ sur $X_{i\ell}$
et des nombres réels strictement positifs $c_{i\ell}
$ tels que $\norm{e_{ij}}|_{X_{i\ell}}=c_{i\ell}\abs{g_{ij\ell}}$ pour tout $j$. La famille des $g_{ij\ell}^{-1}e_{ij}|_{X_{i\ell}}$
est alors une base orthogonale de $E|_{X_{i\ell}}$ dont chaque élément est de norme constante. 
Par conséquent on peut, dans la définition d'un fibré métrisé paralinéaire entier, 
exiger que les $\norm{e_{ij}}$ soient constantes. 
On montre \textit{mutatis mutandis} qu'on peut, dans la définition d'un fibré paralinéaire strictement entier, 
exiger que les $\norm{e_{ij}}$  soient toutes identiquement égales à $1$, c'est-à-dire que les 
bases $(e_{ij})$ soient orthonormées.\end{rema}

\subsection{}
La métrique naturelle
sur le fibré en droites $\mathscr O_X$, pour laquelle $\norm{f}=\abs f$
universellement (c'est-à-dire pour tout morphisme $Y\to X$ et toute $f\in \mathscr O(Y)$) 
est paralinéaire entière, et strictement entière si $X$ est strict. Le fibré
$\mathscr O_X$ muni de cette métrique
sera appelé le fibré en droites métrisé trivial.

Plus généralement, soit $u$ une fonction continue à valeurs réelles sur $X$. 
La métrique sur $\mathscr O_X$ pour laquelle $\norm f=e^{-u}\cdot \abs f$ universellement est alors
paralinéaire (\resp paralinéaire entière, \resp paralinéaire stricte, \resp paralinéaire strictement entière) si et seulement si c'est le cas  
de $u$.

\subsection{}
Soit $E$ et $E'$ des fibrés vectoriels métrisés sur~$X$.
On munit la somme directe $\underline E\oplus \underline E'$ 
des fibrés vectoriels sous-jacents
de la métrique définie universellement
par $(e,e')\mapsto \max(\norm e,\norm {e'})$; on dit que c'est la somme
directe orthogonale et on note $E\oplus E'$ le fibré
vectoriel correspondant.

Cette métrique est paralinéaire (\resp paralinéaire entière, \resp paralinéaire stricte, \resp
paralinéaire strictement entière)
dès que les métriques sur~$E$ et $E'$ le sont.

Un fibré vectoriel métrisé $E$
sur $X$ est paralinéaire strictement entier
si et seulement s'il  existe un \Gstr -reouvrement $(X_i)$ de 
$X$ de $X$
tel que $E|_{X_i}$ soit pour tout $i$ isomorphe 
à une somme directe orthogonale finie de copies du fibré en droites trivial. 

\subsection{}
Soit $L$ et $L'$ deux 
fibrés métrisés paralinéaires sur~$X$; 
les fibrés métrisés 
$L\otimes L'$ et~$L^{-1}$ sont alors paralinéaires.
Un énoncé analogue vaut pour les fibrés métrisés
paralinéaires entiers, paralinéaires stricts, 
et paralinéaires strictement entiers.

Soient maintenant $E$ et $F$ deux fibrés vectoriels métrisés de rangs
arbitraires. 
Supposons que $E$ et $F$ possèdent des bases orthogonales $(e_i)$
et~$(f_j)$; on peut alors munir $E\otimes F$ de la métrique $\norm \cdot$ pour
laquelle la base $(e_i\otimes f_j)$ est orthogonale et vérifie les égalités
$\norm{e_i\otimes f_j}=\norm {e_i}\cdot \norm {f_j}$. 
Cette métrique ne dépend pas du choix des bases orthonormées
fixées $(e_i)$ et $(f_j)$ : c'est en effet l'unique métrique telle que pour
tout $x\in X$, la norme induite sur le $\hr x$-espace
vectoriel $(E\otimes F)(x)=E(x)\otimes F(x)$ soit la 
norme produit tensoriel définie par exemple 
dans \cite[p. 12]{berkovich1990}. 

Cette construction
permet par recollement G-local
ou \Gstr -local de munir le produit tensoriel de deux 
fibrés métrisés paralinéaires (\resp paralinéaires stricts, 
\resp paralinéaires entiers, \resp paralinéaires strictement entiers)
d'une métrique paralinéaire (\resp \ldots ).

De même, si $E$ est un fibré métrisé possédant une base orthogonale $(e_i)$ 
la métrique sur $E^\vee$ pour laquelle la base duale $(e_i^\vee)$ est orthogonale 
et vérifie les égalités $\norm{e_i^\vee}=\norm{e_i}^{-1}$. Cette construction
permet par recollement G-local ou
\Gstr -local de munir le dual d'un fibré métrisé paralinéaire (\resp \ldots) 
d'une métrique paralinéaire (\resp \ldots). 

\subsection{}
Soit $X$ un espace $k$-analytique, soit $Y$ un espace analytique sur une extension complète $L$ de $k$, soit $f\colon Y\to X$
un $k$-morphisme et soit $E$
un fibré vectoriel métrisé sur $X$. Si $E$ est paralinéaire (\resp paralinéaire entier) il en va de même de $f^*E$. 
Si $X$ et $Y$ sont stricts et si $E$ est paralinéaire strict (\resp paralinéaire strictement entier), 
il en va de même de $f^*E$.

\subsection{}\label{ss-sleqr-domaine}
Soit $E$ un fibré muni d'une métrique paralinéaire, soit $s$ une section de
$E$ et soit $r$ un réels strictement positif. L'ensemble des point $x$
de $X$ tels que $\|s(x)\|\leq r$ est alors un domaine analytique de $X$. 
En effet, on peut pour le voir raisonner G-localement
sur $X$, et donc supposer que $E$ possède une base orthogonale $(e_i)$. 
On peut alors écrire $s=\sum a_i e_i$ où les $e_i$ sont des fonctions
analytiques sur $X$. Par définition d'une métrique paralinéaire il existe
une famille $(g_i)$ de fonctions à valeurs réelles sur $X$, dont les logarithmes
sont paralinéaires, et telles que $\|s\|=\sup_i \abs{a_i}\cdot g_i$ identiquement
sur $X$. L'assertion s'en déduit immédiatement. 

Si $X$ est strict, si $E$ est strictement paralinéaire et si $r$
est de torsion modulo $\abs{k^\times}$ on démontre \textit{mutatis mutandis}
que l'ensemble des point $x$
de $X$ tels que $\|s(x)\|\leq r$ est un domaine strictement 
analytique de $X$.

\section{Fibré résiduel d'un fibré paralinéaire entier}

Soit $X$ un espace $k$-analytique séparé (au sens analytique), 
soit $x$ un point de $X$
et soit $E$ un fibré vectoriel métrisé paralinéaire entier sur $X$  dont
on note $n$ le rang.

\subsection{}
La tige
$E(x)$ est un $\hr x$-espace vectoriel normé
de dimension finie. Elle possède à ce titre un espace vectoriel 
gradué résiduel $\widetilde {E(x)}$ sur le corps 
gradué $\hrt x$ (\ref{ss-red-graduee}). 

Puisque $E$ est paralinéaire entier, il existe un domaine
analytique $V$ de $X$ contenant $x$, une base orthogonale
$(e_1,\ldots, e_n)$ de $E|_V$ et une famille
$(r_1,\ldots, r_n)$ de réels strictement positifs
telle que $\|e_i\|=r_i$ identiquement sur $V$ pour tout $i$
(remarque \ref{rem-bases-orthogonales}). 
Les $e_i(x)$ forment alors
une base du $\hr x$-espace
vectoriel normé 
$E(x)$ et l'on a 
$\|\sum \lambda_i e_i(x)\|=
\sup_i \abs{\lambda_i}r_i$ pour toute
famille $(\lambda_i)$ d'éléments de
$\hr x^n $. Il s'ensuit que  les $\widetilde {e_i(x)}$ forment une base
du $\hrt x$-espace vectoriel gradué 
$\widetilde{E(x)}$ (\ref{ss-independance-lin-residuelle}). 
En particulier, $\widetilde {E(x)}$ est de dimension $n$. 

\begin{lemm}\label{lem-exist-sections-locales}
Supposons que
$X$ est presque bon en $x$ et que $E$ est engendré par ses sections globales
sur $X$ en tant que $\mathscr O_X$-module.

\begin{enumerate}
\item Soit $\sigma$ un élément homogène de $\widetilde {E(x)}$.
Il existe un voisinage analytique $U$ de $x$ dans $X$ et une section 
$s$ de $E$ sur $U$ telle que $\sigma=\widetilde{s(x)}$. 

\item Soit $F$ un autre fibré vectoriel métrisé paralinéaire entier
sur $X$ engendré par ses
sections globales comme $\mathscr O_X$-module, 
et soit $\theta$ un isomorphisme de $\widetilde {E(x)}$ sur $\widetilde {F(x)}$. Il existe
un voisinage analytique $U$ de $x$ dans $X$ et un isomorphisme $u\colon E|_U\to F|_U$ de fibrés vectoriels
tel que l'isomorphisme $u(x)\colon E(x)\to F(x)$ soit une isométrie dont
la réduction $\widetilde{u(x)}$ est égale à $\theta$. 
\end{enumerate}
\end{lemm}

\begin{proof}
Montrons d'abord (a). 
L'assertion est évidente si $\sigma$ est nulle. Supposons $\sigma\neq 0$ et
soit $r$ son degré. Il existe alors un élément $t$ de $E(x)$ tel que
$\|t\|=r$ et $\sigma=\widetilde t$. Comme $E$ est engendré par ses sections
globales, il existe $s_1,\ldots, s_m$ dans $E(X)$ et des éléments
$\lambda_1,\ldots, \lambda_m$ de $\hr x$ tels que
$t=\sum \lambda_i s_i(x)$. Puisque $X$ est presque bon en $x$, il existe un
voisinage analytique $U$ de $x$ et des fonctions analytiques $\mu_1,\ldots, \mu_m$
sur $U$ telles que $\abs{\lambda_i-\mu_i(x)}\cdot \|s_i(x\|<r$ pour tout $i$. 
Notons $s$ la section $\sum \mu_i s_i|_U$ de $E$ sur $U$. On a par 
construction $\|s(x)-t\|<r$, si bien que $\widetilde{s(x)}=\widetilde
t=\sigma$. 

Montrons maintenant (b).
Choisissons uen base de $\widetilde{E(x)}$. 
Il résulte de (a) qu'il existe
une famille $(e_1,\ldots, e_n)$de sections de $E$ 
sur un voisinage analytique $U$ de $x$ 
telle que la base en question soit égale à $(\widetilde{e_i(x)})$. 
On a alors l'égalité
$\abs{\sum \lambda_i e_i(x)}=\sup_i \abs{\lambda_i}\|e_i(x)\|$ pour toute
famille $(\lambda_i)$ d'éléments de $\hr x$ 
(\ref{ss-independance-lin-residuelle}), ce qui implique que la famille des
$e_i(x)$ est libre, puis qu'elle est une base de $E(x)$ pour des raisons 
de dimension, et finalement 
une base de $E$ au voisinage de $x$ puisque $E$ est trivialisable au voisinage 
de $x$. On peut donc restreindre $U$ de sort que les $e_i$
forment une base de $E|_U$. 

Comme
$\theta$ est un isomorphisme, la famille des 
$\theta(\widetilde{f_i(x)})$ est une base de $\widetilde{F(x)}$. 
En utilisant à nouveau (a) et en raisonnant comme ci-dessus on voit que quitte à restreindre
$U$ on peut supposer qu'il existe une base $(f_i)$ de $F|_U$ telle que 
$\widetilde{f_i(x)}=\theta(\widetilde{u_i(x)})$ pour tout $i$, et que l'on a 
$\abs{\sum \lambda_i f_i(x)}=\sup_i \abs{\lambda_i}\|f_i(x)\|$ pour toute
famille $(\lambda_i)$ d'éléments de $\hr x$ ; 
remarquons que comme $\theta$ préserve les degrés on a $\|f_i(x)\|=\|e_i(x)\|$.
Soit $u$ l'isomorphisme de $E|_U$ sur $F|_U$ qui envoie $e_i$ sur $f_i$
pour tout $i$. 
Comme on a pour toute famille 
$(\lambda_i)$ d'éléments de $\hr x$ les égalités
\[\left|u\left(\sum \lambda_i e_i(x)\right)\right|=\left|\sum \lambda_i f_i(x)\right|=
\sup_i \abs{\lambda_i}\cdot \|f_i(x)\|
=\sup_i \abs{\lambda_i}\cdot \|e_i(x)\|=\left|\sum \lambda_i e_i(x)\right|,\]
le morphisme $u(x)$ est bien une isométrie, 
et comme $\widetilde{u(x)}$ envoie
$\widetilde{e_i(x)}$ sur $\widetilde{f_i(x)}$ par construction il vient
$\widetilde{u(x)}=\theta$, d'où (b). 
\end{proof}

\begin{prop}\label{prop-metrique-residuelle}
Soit $\sigma$ une section homogène non nulle de
$\widetilde E(x)$, soit $r$ son degré
et soit $v$ un élément de
$\widetilde{(X,x)}$.

\begin{enumerate}

\item Il existe un domaine analytique $V$ de $X$ contenant $x$
tel que $v\in \widetilde {(V,x)}$,
une base orthogonale $(e_i)$ de $Ê|_V$ telle que $\|e_i\|$ soit
constante pour tout $i$, et 
une section $s$ de $E$ sur $V$ vérifiant 
l'égalité $\widetilde{s(x)}=\sigma$. 
\item Soit $(V,(e_i), s)$ comme en (a) ; pour 
tout $i$, notons $r_i$ la valeur constante de $\|e_i\|$
sur $V$. 
Écrivons
$s=\sum a_i e_i $ et désignons par
$I$ l'ensemble des indices
$i$ tels que $\abs{a_i(x)}\cdot r_i=r$. L'élément $\sup_{i\in I}v(
\widetilde{a_i(x)})$ de
$v(\hrt x^\times)$ ne dépend alors pas du choix
de $(V,(e_i), s)$. Notons-le 
$v(\sigma)$. 

\item Soit $W$ un domaine analytique de $X$
contenant $x$
et tel que $v\in \widetilde{(W,x)}$, 
soit $a$ une fonction analytique 
sur $W$ ne s'annulant pas en $x$
et soit $t$ une section de $E$ sur $W$
telle que $\widetilde {t(x)}=\sigma$. 
Les assertions suivantes sont équivalentes : 
\begin{itemize}
\item[(i)] $v(\widetilde{a(x)})v(\sigma)\leq 1$ ; 
\item[(ii)] le domaine analytique 
$W'$ de $W$ défini par l'inégalité
$\|at\|\leq r\cdot \abs{a(x)}$ vérifie la condition 
$\widetilde{(W',x)}\ni v$. 
\end{itemize}
\end{enumerate}
\end{prop}
Notons que $W'$ est bien un domaine
analytique de $W$ d'après
\ref{ss-sleqr-domaine}. 

\begin{proof}
Comme  la métrique de $E$ est paralinéaire entière, il existe 
un G-recouvrement $(V_i)$ de $X$ tel que $E|_{V_i}$ possède pour tout $i$
une base orthogonale dont les vecteurs sont de norme constante 
(\ref{rem-bases-orthogonales}). Quitte à raffiner ce recouvrement, on peut 
supposer que les $V_i$ sont affinoïdes. Soit $I$ l'ensemble des indices
tels que $x\in V_i$. Les $(V_i,x)$ pour $i\in I$ constituent un G-recouvrement
du germe $(X,x)$, si bien que $\widetilde{(X,x)}$ est la réunion des
$\widetilde{(V_i,x)}$. Il existe donc $i\in I$ tel que $v\in 
\widetilde{(V_i,x)}$. Et comme $V_i$ est affinoïde et en particulier presque
bon en $x$, il résulte du lemme \ref{lem-exist-sections-locales} 
qu'il existe un voisinage analytique $V$ de $x$ dans  $V_i$ et une section 
$s$ de $E$ sur $V$ telle que $\widetilde{s(x)}=\sigma$, ce qui achève 
de prouver (a). 

Prenons $V,(e_i), (r_i), s,(a_i), I$ comme dans (b), 
et posons $v(\sigma)=\max_{i\in I}v(
\widetilde{a_i(x)})$. Avant de montrer que que $v(\sigma)$
ne dépend effectivement que de $v$ et de $\sigma$, nous allons montrer qu'il
satisfait (c). Soit donc $W$ et $t$ comme dans (c). Pour établir l'équivalence
entre (i) et (ii) on peut remplacer $W$ par n'importe lequel de ses domaines
contenant $x$ et dont la réduction contient $v$, et en particulier
par $V\cap W$ ; et l'on peut aussi remplacer $V$ par $V\cap W$, ce qui 
permet finalement de supposer que $W=V$. 
Les égalités $\widetilde {t(x)}=\sigma=\widetilde{s(x)}$ entraînent que
$\|s(x)-t(x)\|<r$ si bien que pour toute fonction 
$a$ sur $W$ non nulle en $x$ les inégalités 
$\|as\|\leq r\abs{a(x)}$ et 
$\|at\|\leq r\abs{a(x)}$ définissent le même domaine analytique au voisinage
de $x$ dans $V$. On peut donc également pour établir l'équivalence
voulue remplacer $t$ par $s$, c'est-à-dire supposer que $t=s$. 
On a $\|s\|=\max_i \abs{a_i}r_i$ identiquement sur $V$ ; quitte à remplacer $V$
par un voisinage convenable de $x$ dans $V$, on peut supposer que
 $\|s\|=\max_{i\in I} \abs{a_i}r_i$ identiquement sur $V$. 

Soit $a$ une fonction analytique sur $V$ ne s'annulant pas en $x$. 
Pour harmoniser les notations, désignons par $V'$ (plutôt que $W'$) le domaine
analytique 
de $V$ défini par l'inégalité $\|as\|\leq r\abs{a(x)}$. 
On peut également définir $V'$ par la conjonction d'inégalités
\[\bigwedge_{i\in I} \abs{aa_i}r_i\leq r\cdot \abs{a(x)}=\abs{a(x)a_i(x)}\cdot r_i,\]
c'est-à-dire encore par la conjonction d'inégalités 
\[\bigwedge_{i\in I} \abs{aa_i}\leq \abs{a(x)a_i(x)}.
\]
Il s'ensuit que $\widetilde{(V',x)}$ est égal 
à $\zr{\hrt x}{\widetilde k}\{\widetilde{a_i(x)}\widetilde{a(x)}\}_{i\in I}$. 
Par conséquent, $v$ appartient à $\widetilde{(V',x)}$
si et seulement si $\max_{i\in I}v(\widetilde{a_i(x)})v(\widetilde{a(x)}\leq 1$,
c'est-à-dire encore si et seulement si 
$v(\widetilde{a(x)}v(\sigma)\leq 1$, ce qui montre (c). 

Il reste à prouver (b). Pour ce faire, fixons un domaine
affinoïde $W$ de $X$ contenant $x$ et une section $t$
de $E$ sur $W$ telle que $\widetilde{t(x)}=\sigma$ (l'existence d'un tel
$(W,t)$ est assuré par (a)). 

Soit $(V,(e_i), s, (a_i))$ comme en (b) ; posons
$v(\sigma)=\max_{i\in I}v(
\widetilde{a_i(x)})$. 
Soit $\gamma$ un élément de $v(\hrt x^\times)$. 
Comme $W$ est affinoïde, il existe une fonction analytique $a$
définie sur un voisinage analytique $U$ de $x$ dans $W$
tel que $v(\widetilde{a(x)})=\gamma$. Soit $U'$ le domaine
analytique de $U$ défini par l'inégalité 
$\|at\|\leq r\cdot \abs{a(x)}$. D'après l'assertion 
(c) déjà établie, on a $\gamma v(\sigma)\leq 1$
si et seulement si $v$ appartient à $\widetilde{(U',x)}$. 
Cette dernière condition ne dépend pas du choix de $(V,(e_i), s)$. 
Par conséquent, l'ensemble $\Gamma$ des éléments $\gamma$ de $v(
\widetilde{a_i(x)})$ tels que $\gamma v(\sigma)\leq 1$ est 
indépendant de $(V,(e_i), s)$ ; or $v(\sigma)$ est l'inverse du plus grand
élément de $\Gamma$, ce qui montre (b). 
\end{proof}

\subsection{}\label{ss-def-metrique-residuelle}
La donnée pour tout élément homogène non nul
$\sigma$ de $\widetilde {E(x)}$ de l'application $v\mapsto v(\sigma)$
définit (en posant en plus $v(0)=0)$ une métrique au-dessus
de $\widetilde{(X,x)}$
sur le
$\hrt x$-espace vectoriel gradué $\widetilde {E(x)}$
(\ref{def-metrique-graduee}).

Soit $V$ un domaine analytique de $X$ contenant $x$ et sur lequel
$E$ possède une base orthogonale $(e_i)$
dont les vecteurs sont de norme constante.  
La famille des $\widetilde{e_i(x)}$ est alors une base
de $\widetilde{E(x)}$ (\ref{ss-independance-lin-residuelle}), et il résulte 
de l'assertion (b) de la proposition 
\ref{prop-metrique-residuelle}
que cette base est orthonormée 
au-dessus de $\widetilde{(V,x)}$.

Puisque $E$ est paralinéaire entier, il existe une famille finie
$(V_i)$ de domaines analytiques contenant $x$ et dont la réunion est un 
voisinage de $x$ tel que chacun des $V_i$ possède une base orthogonale
constituée de vecteurs de norme constante. 
Les
$\widetilde{(V_i,x)}$ recouvrent $\widetilde {(X,x
)}$, et il résulte de ce qui précède que $\widetilde{E(x)}$ possède pour tout $i$
une base orhonormée au-dessus de $\widetilde{(V_i,x)}$. 
Ceci entraîne que
$(\widetilde {E(x)}, s\mapsto (v\mapsto v(s))$ est associé à 
un fibré vectoriel de rang $n$ sur $\widetilde{(X,x)}$, 
que nous noterons $\widetilde{E_x}$. Nous l'appelons
le \textit{fibré résiduel} de $E$ en $x$. 

\subsection{}\label{rappels-fibgradue-metriques}
Rappelons ici les liens entre $\widetilde{E_x}$ et $\widetilde{E(x)}$ : 
\begin{itemize}
\item le faisceau $\widetilde{E_x}$ est un sous-faisceau du faisceau constant
$\underline{\widetilde{E(x)}}$ ; pour tout ouvert non vide $U$ de $\widetilde {(X,x)}$ et tout $r>0$ le groupe
$\widetilde{E_x}(U)_r$ est le sous-groupe de $\widetilde{E(x)}_r$ formé des sections
$s$ telles que $v(s)\leq 1$ pour tout $v\in U$ ;
\item l'espace vectoriel gradué $\widetilde{E(x)}$ est la fibre générique du fibré vectoriel
$\widetilde{E_x}$ ; pour toute section homogène $s$ de $\widetilde{E(x)}$, tout
$v\in U$ et tout $a\in \hrt x^\times$ on a $v(s)\leq v(a)^{-1}$ si et seulement si 
$as\in (\widetilde{E_x})_v$ ; 
\item si $(\eps_1,\ldots, \eps_n)$ est une base de $\widetilde{E(x)}$ et si $U$ est un ouvert de
$\widetilde{(X,x)}$, la base $(\eps_i)$ est orthonormée au-dessus de $U$ si et seulement si 
les $\eps_i|_U$ constituent une base de $\widetilde{E_x}|_U$. 
\end{itemize}

On déduit alors de \ref{ss-def-metrique-residuelle} que si $V$ est un domaine 
analytique de $X$ contenant $x$ tel que $E|_V$ possède une base orthognale 
$(e_i)$ dont les vecteurs sont de norme constante sur $V$, les $\widetilde{e_i(x)}|_{\widetilde{(V,x)}}$ 
forment une base de $\widetilde{E_x}|_{\widetilde{(V,x)}}$.

\subsection{}
Supposons que $X$ est presque bon en $x$ et que 
$E$ est trivialisable comme fibré vectoriel 
au voisinage de $x$. Soit $V$
un domaine analytique de $X$ contenant $x$. 
Il résulte du lemme \ref{lem-exist-sections-locales}
que pour tout réel $r>0$ le groupe des sections
homogènes de degré $r$ de $\widetilde{E_x}$ sur $\widetilde{(V,x)}$
est le quotient du groupe des sections $s$ de $E$
sur le germe $(X,x)$  telles que $\|s\|\leq r$ 
sur le germe $(V,x)$
par son sous-groupe formé des sections $s$ 
telles que $\|s(x)\|<r$.

\section{Propriétés du foncteur de réduction des fibrés métrisés}

On conserve les notations de la section précédente : $X$ est un espace $k$-analytique séparé, $x$ un point de $X$
et $E$ un fibré vectoriel métrisé paralinéaire entier sur $X$  dont
on note $n$ le rang.

\subsection{}\label{fibre-res-foncteur}
Il résulte de la construction que ni la métrique sur $\widetilde{E(x)}$
au-dessus de $\widetilde{(X,x)}$ ni (par conséquent) $\widetilde{E_x}$
ne sont modifiés si l'on remplace $X$ par un voisinage 
analytique $U$ de $x$ dans $X$, et $E$ par $E|_U$ : le fibré résiduel de
$E$ en $x$ ne dépend donc que du germe de $E$ en $x$. 

Cette dépendance est fonctorielle : si $U$ est un voisinage analytique de $x$
dans $X$ et si $F$ un fibré vectoriel sur $U$ muni d'une métrique
paralinéaire entière, toute isométrie de $E|_U$ sur $F$ induit une isométrie
de $\widetilde{E(x)}$ sur $\widetilde{F(x)}$, et donc un isomorphisme
de $\widetilde{E_x}$ sur $\widetilde{F_x}$. 

Si $V$ est un domaine analytique de $X$ contenant $x$, on a 
par construction $\widetilde{(E|_V)_x}=\widetilde{E_x}|_{\widetilde{(V,x)}}$. 
Plus généralement, pour tout espace analytique séparé 
$Y$ sur une extension de $k$, 
tout morphisme $p\colon Y\to X$ et tout $y\in p^{-1}(x)$, le fibré résiduel 
$\widetilde{(p^*E)_y}$ est égal à 
$q^*\widetilde{E_x}$ où $q$ est l'application 
de $\widetilde{(Y,y)}$ vers $\widetilde{(X,x)}$ induite par $p$. 

\subsection{}\label{fibre-res-tenseur}
Soit $F$ un fibré métrisé paralinéaire entier
sur $X$. Soit $G$ le fibré métrisé $E\otimes F$. 
Soit $\sigma$ un élément homogène de $\widetilde{E(x)}$ et soit $\tau$ un élément
homogène de $\widetilde{F(x)}$. Relevons-les respectivement en des éléments
$s$ et $t$ de $E(x)$ et $F(x)$. Un calcul explicite dans des bases orthogonales
de $E(x)$ et $F(x)$ montre que $\widetilde{(s\otimes t)(x)}$ ne dépend que de $\sigma$ et 
$\tau$. On définit ainsi une application bilinéaire de $\widetilde{E(x)}\times \widetilde{F(x)}$ vers
$\widetilde {G(x)}$ et donc une application linéaire de $\widetilde {E(x)}\otimes \widetilde{F(x)}$
vers $\widetilde{G(x)}$ dont on vérifie (en utilisant là encore des bases orthogonales de $E(x)$ et $F(x)$)
que c'est un isomorphisme.Cet isomorphisme est en fait une isométrie, comme on le voit en se ramenant
par un raisonnement G-local sur le germe $(X,x)$ au cas où $E$ et $F$ possèdent chacun une base orthogonale
dont les vecteurs sont de norme constante. 

On dispose donc d'un isomorphisme naturel entre $\widetilde{E_x}\otimes\widetilde{F_x}$ et $\widetilde{G_x}$. 

\subsection{}\label{fibre-res-dual}
Soit $r>0$ et soit $\sigma$ une section de $\widetilde{E^\vee(x)}^\vee$ homogène de degré $r$. 
Soit $\tau$ une section homogène de $\widetilde{E(x)}$. Relevons $\sigma$ en un élément
$s$ de $\mathrm{Hom}(E(x), \hr x[r])$ et $\tau$ en un élément 
$t$ de $E(x)$. Un calcul explicite dans une base orthogonale de $E(x)$ montre que 
l'élément homogène $\widetilde{s(t)(x)}$ de $\hrt x[r]$ ne dépend que de $\sigma$ et $\tau$, 
et pas des relèvements $s$ et $t$. On obtient ainsi une application linéaire de $\widetilde{E^\vee(x)}$ dans
$\widetilde{E(x)}^\vee$ dont on vérifie (en utilisant là encore une base orthogonale de $E(x)$)
que c'est un isomorphisme.Cet isomorphisme est en fait une isométrie, comme on le voit en se ramenant
par un raisonnement G-local sur le germe $(X,x)$ au cas où $E$ possèdent une base orthogonale
dont les vecteurs sont de norme constante. 

On dispose donc d'un isomorphisme naturel entre $\widetilde{E^\vee_x}$
et $\widetilde{E_x}^\vee$.

\begin{lemm}\label{lem-isometrie-residuelle}
Soit $F$ un fibré vectoriel métrisé paralinéaire entier 
sur $X$. 
Soit $V$ un domaine analytique de $X$ contenant $x$.
Soit $u$ un isomorphisme de fibrés vectoriels de $E$ sur $F$ tel que
le morphisme induit $u(x)\colon E(x)\to F(x)$ soit 
une isométrie ; soit
$\theta$ l'isomorphisme de 
$\widetilde {E(x)}$ sur $\widetilde{F(x)}$ induit par $u(x)$. 
Les assertions suivantes sont équivalentes : 

\begin{itemize}
\item[(i)] il existe un voisinage analytique $U$
de $x$ dans $X$ tel que $u|_{U\cap V}$ soit une isométrie ; 

\item[(ii)] l'application $\theta$ est une isométrie au-dessus de $\widetilde{(V,x)}$, c'est-à-dire que 
l'on a $v(\theta(s))=v(s)$ pour toute section homogène $s$
de $\widetilde {E(x)}$ et toute $v\in \widetilde{(V,x)}$ ; 
\item[(iii)] l'application $\theta$ induit un isomorphisme de
$\widetilde{E_x}|_{\widetilde{(V,x)}}$ sur 
$\widetilde{F_x}|_{\widetilde{(V,x)}}$. 
\end{itemize}
\end{lemm}

\begin{proof}
L'équivalence de (ii) et (iii) résulte des généralités rappelées en 
\ref{rappels-fibgradue-metriques}, 
et l'implication (i)$\Rightarrow$(ii) provient de la fonctorialité
de la construction du fibré résiduel, et de son bon
comportement par restriction à un domaine analytique. 
Supposons maintenant que (ii) est vraie, et 
montrons (i). 
Pour ce faire on peut remplacer $X$
par n'importe quel voisinage analytique de $x$ dans $X$. 
On peut ainsi supposer que $X$ et $V$ sont compacts. 
Il existe alors un recouvrement fini $(V_j)$ de $V$ 
par des domaines analytiques compacts tels que $E|_{V_j}$ et $F|_{V_j}$
admettent pour tout $j$ une base orthogonale dont les vecteurs sont de norme
constante sur $V_j$. Il suffit alors d'exhiber pour chaque $j$ un voisinage ouvert
$U_j$ de $x$ dans $X$ tel que $u|_{V_j\cap U_j}$ soit une isométrie : en effet dans ce cas
la restriction de $u$ à $U:=\bigcap U_j$ induira une isométrie sur 
$V\cap U$, par vérification sur le G-recouvrement de
$V\cap U$ par les 
$V_j\cap U$.  On se ramène ainsi au cas où 
$E|_V$ et $F|_V$ possèdent chacun une base orthogonale à vecteurs de norme
constante. Il suffit dès lors de montrer l'existence d'un 
voisinage analytique $U$ de $x$ tel que $u|_{V\cap U}$ diminue
les normes (au sens large). En effet l'application de  ce résultat à $u^{-1}$ 
montrera que $U$ peut être choisi de sorte que $u^{-1}|_{V\cap U}$ diminue 
également les normes, auquel cas $u|_{V\cap U}$ sera une isométrie. 

Choisissons une base orthogonale $(s_i)$ pour de $E|_V$ 
tel que $\|s_i\|$ soit pour tout $i$ de norme constante $\rho_i$. 

Par hypothèse sur $u$ on a pour tout indice $i$ les
égalités $\|u(s_i)(x)\|=\|s_i(x)\|=\rho_i$
et 
$\widetilde{u(s_i)(x)}=\theta(\widetilde{s_i(x)})$. 
Soit $v\in \widetilde {(V,x)}$. Par 
choix de $s_i$ on a $\|s_i\|=\rho_i$ identiquement
sur $V$, si bien que $v(\widetilde{s_i(x)})=1$. Comme
on travaille sous l'hypothèse (ii) on en déduit que
$v(\theta(\widetilde{s_i(x)}))=1$. Ceci valant pour tout $v$
appartenant à $\widetilde{(V,x)}$, on a l'égalité 
$\|u(s_i)\|=\rho_i$ au voisinage de $x$ sur $V$. Comme il y a un nombre fini d'indices $i$,
on en déduit l'existence d'un voisinage analytique $U$
de $x$ dans $X$ tel que $\|u(s_i)\|=\rho_i$ identiquement sur $V\cap U$
pour tout $i$. 
Pour toute famille $(b_i)$ de fonctions analytiques définies sur un 
domaine analytique $W$ de $V\cap U$ on a alors
$\|\sum b_i s_i\|=\max_i \|b_i\|\rho_i$ identiquement sur $W$, 
et \[\|u(\sum b_i s_i)\|=\|\sum b_i u(s_i)\|\leq \max_i \|b_i\|\rho_i\]
identiquement sur $W$. Par
conséquent $u$ diminue les normes sur $V\cap U$, ce qui achève la démonstration. 
\end{proof}

\begin{coro}\label{coro-BON}
Soit $(e_i)$ une base de $E$ telle que les $\widetilde{e_i(x)}$ forment une base
de $\widetilde{E(x)}$ et soit $V$ un domaine analytique de $X$ contenant $x$. 
Les assertions suivantes sont équivalentes : 

\begin{itemize}
\item[(i)] il existe un voisinage analytique $U$ de $x$ tel que les
$e_i|_{V\cap U}$ soient de norme constante et constituent une base orthogonale
de $E|_{V\cap U}$ ; 

\item[(ii)] la base $(\widetilde{e_i(x)})$
est orthonormée au-dessus de $\widetilde{(V,x)}$ ;

\item[(iii)] les $\widetilde{e_i(x)}|_{\widetilde{(V,x)}}$ constituent une base de $\widetilde{E_x}
|_{\widetilde{(V,x)}}$. 

\end{itemize}
\end{coro}

\begin{proof}
L'équivalence de (ii) et (iii) résulte des propriétés générales rappelées en \ref{rappels-fibgradue-metriques}. 
L'implication (i)$\Rightarrow$(ii) découle de \ref{ss-def-metrique-residuelle}. 

Supposons que (iii) soit vraie. 
Soit $n$ le rang de $E$. 
Pour tout $i$, posons $r_i=\|e_i(x)\|$. 
Soit $F$ le fibré vectoriel $\mathscr O_X^n$ muni de la 
métrique paralinéaire entière $(a_i)\mapsto \sup_i \|a_i\|\cdot r_i$. 
Soit $u$ le morphisme de $F$ vers $E$ qui envoie $(a_i)$ sur $\sum a_i e_i$. 
L'hypothèse faite sur les $e_i$ signifie précisément que $u(x)\colon F(x)\to E(x)$ est une isométrie
(\ref{ss-independance-lin-residuelle}). 
Soit $(f_i)$ la base canonique de $F$. Elle est par construction
orthogonale et constituée de vecteurs de norme constante. 
En utilisant (pour le fibré métrisé paralinéaire entier $F$) l'implication  (i)$\Rightarrow$(iii) déjà établie, on en déduit que les $\widetilde{f_i(x)}$
forment une base
de $\widetilde{F_x}$. Puisqu'on suppose que (iii) est vraie (pour le fibré métrisé paralinéaire entier $E$), le morphisme
$\widetilde{u(x)}$ induit un isomorphisme de $\widetilde{F_x}|_{\widetilde{(V,x)}}$ sur  $\widetilde{E_x}|_{\widetilde{(V,x)}}$. 
Le lemme \ref{lem-isometrie-residuelle} assure alors l'existence d'un voisinage analytique $U$ de $x$ dans $X$
tel que $u|_{V\cap U}$ soit une isométrie ; mais cela signifie précisément que les $e_i|_{V\cap U}$ sont
de norme constante et 
constituent une base orthogonale de $E|_{V\cap U}$. 

\end{proof}

\begin{coro}\label{coro-releve-isometrie}
Supposons que $X$ est presque bon en $x$ et que $E$ est trivialisable au voisinage de $x$ comme fibré
vectoriel. Soit $F$ un fibré vectoriel métrisé paralinéaire entier sur $X$, lui aussi trivialisable comme fibré vectoriel
au voisinage de $X$. Soit $V$ un domaine analytique de $X$ contenant $x$ et soit $\theta$ un isomorphisme
de $\widetilde{E(x)}$ sur $\widetilde{F(x)}$ induisant un isomorphisme
de $\widetilde{E_x}|_{\widetilde{(V,x)}}$ sur $\widetilde{F_x}|_{\widetilde{(V,x)}}$. Il existe alors un voisinage
analytique $U$ de $x$ dans $X$ et un isomorphisme de fibrés vectoriels $u\colon E_U\simeq F|_U$ tel que $u|_{V\cap U}$ soit une 
isométrie et tel que $\widetilde{u(x)}=\theta$. 
\end{coro}

\begin{proof}
Il existe par hypothèse un voisinage analytique $U$ de $x$ dans $X$ tel que $E|_U$ et $F|_U$ soient libres et en particulier
engendré par leurs sections globales. Le lemme \ref{lem-exist-sections-locales} assure alors qu'on peut supposer, quitte
à restreindre $U$, qu'il existe un isomorphisme $u$ de fibrés vectoriels de $E|_U$ sur $F|_U$ tel que $u(x)$ soit une isométrie
dont la réduction $\widetilde{u(x)}$ est égale à $\theta$. 
Il résulte alors du lemme \ref{lem-isometrie-residuelle} qu'on peut restreindre $U$ de sorte que 
$u|_{V\cap U}$ soit une isométrie.
\end{proof}

\begin{prop}\label{prop-fibreres-essurj}
Supposons que $X$ est presque bon en $x$. Soit $M$ un fibré vectoriel
sur $\widetilde {(X,x)}$. Il existe un voisinage analytique $U$ de
$x$ dans $X$ et un fibré vectoriel métrisé paralinéaire entier $F$ sur $X$, 
trivialisable au voisinage de $x$ comme fibré vectoriel, tel que
$\widetilde{F_x}\simeq M$.
\end{prop}

\begin{proof}
Soit
$m$ le rang de $M$. 
Il existe une famille finie $(V_\alpha)$ de domaines analytiques contenant
$x$ tel que les germes $(V_\alpha,x)$ recouvrent $(X,x)$ et, pour tout $i$, 
un isomorphisme $\theta_\alpha$
de $ M|_{\widetilde{(V_\alpha,x)}}$ sur $ \mathscr 
O_{\widetilde{(V_\alpha,x)}}[r_\alpha]$ pour un certain polyrayon  $r_\alpha=(r_{\alpha i})_i$ de longueur
$m$. Nous notons encore 
$\theta_\alpha$ l'isomorphisme $M_\eta\simeq \hrt x[r_\alpha]$ d'espaces vectoriels
gradués induit par localisation au point générique. 

Pour tout $\alpha$, désignons par $F_\alpha$ le fibré vectoriel 
$\mathscr O_X^m$ muni de la métrique paralinéaire entière
$(a_i)\mapsto \sup_i \abs{a_i}\cdot r_{\alpha i}$ ; on a 
$\widetilde{F_{\alpha,x}}=\mathscr O_{\widetilde{(X,x)}}[r_\alpha]$. 

Fixons un indice 
$\alpha$. 
En vertu du corollaire
\ref{coro-releve-isometrie} et de la finitude de la famille
$(V_\beta)_\beta$ 
il existe 
un voisinage analytique $U$ de $x$ dans $X$ et, pour tout 
$\beta\neq \alpha$, un automorphisme $u_{\alpha\beta}$ de $\mathscr  O_U^m$ 
induisant une isométrie de 
$F_\beta|_{V_\alpha\cap V_\beta\cap U}$ sur $F_\alpha|_{V_\alpha\cap V_\beta\cap U}$ ; 
on pose par ailleurs $u_{\alpha\alpha}=\mathrm{Id}_{\mathscr O_U}^m$.

Pour tout couple $(\beta,\gamma)$ on pose alors
$u_{\beta\gamma}=u_{\alpha\beta}^{-1}u_{\alpha \gamma}$ (la notation
$u_{\alpha\gamma}$ dans ce sens coïncide avec celle introduite précédemment). 
Chacun des $u_{\beta\gamma}$ est un automorphisme de $\mathscr O_U^m$ induisant 
une isométrie 
de
$F_\gamma|_{V_\beta\cap V_\gamma\cap U}$ sur $F_\beta|_{V_\beta\cap V_\gamma\cap U}$. La famille des $(u_{\beta \gamma})$
satisfait visiblement les relations de cocycle et l'on peut donc l'utiliser pour recoller les fibrés métrisés
$F_\beta|_{U\cap V_\beta}$ en un fibré métrisé
$F$ ; par définition, $F$ est paralinéaire entier, et il résulte de sa construction que l'on dispose pour tout $\beta$
d'un isomorphisme composé
\[\lambda_\beta \colon \widetilde{F_x}|_{\widetilde{(V_\beta,x)}}
\simeq\widetilde{F_{\beta,x}}|_{\widetilde{(V_\beta,x)}}=
\begin{tikzcd}\mathscr 
O_{\widetilde{(V_\beta,x)}}[r_\beta]\ar[rr,"\sim", "\theta_\beta^{-1}"'] &&M|_{\widetilde{(V_\beta,x)}}
\end{tikzcd}\] et que 
$\lambda_\beta=\lambda_\gamma$ sur $\widetilde{(V_\beta,x)}\cap \widetilde{(V_\gamma,x)}$
pour tout $(\beta,\gamma)$. Ainsi, les $\lambda_\beta$ se recollent en un isomorphisme de
$\widetilde{F_x}$ sur $M$. 
\end{proof}

\subsection{}\label{ss-categories-germes-fibres}
Soit $X$ un espace $k$-analytique séparé et soit $x\in X$. 
Notons $\mathsf F(X,x)$ le groupoïde des germes de fibrés vectoriels paralinéaires entiers en $x$, défini comme suit. 
On part de la catégorie $\mathsf C (X,x)$ dont les objets sont les couples $(U,E)$ où $U$ est un voisinage
analytique de $x$ dans $X$ et $E$ un fibré vectoriel métrisé paralinéaire entier sur $U$, et dans laquelle l'ensemble
des morphismes de $(V,F)$ vers $(U,E)$ est vide si $V$ n'est pas contenu dans $U$, et est l'ensemble des isométries
de $F$ sur $E|_U$ sinon ; on définit alors $\mathsf F(X,x)$ comme la localisation de $\mathsf C(X,x)$ par la collection 
de toutes 
ses flèches ; un objet  de $\mathsf F(X,x)$ possède un rang (le rang en $x$ du fibré vectoriel d'où il provient), 
invariant par isomorphismes de $\mathsf F(X,x)$. On note $\breve {\mathsf F}(X,x)$ la sous-catégorie de $\mathsf F(X,x)$
formée des objets qui sont triviaux en tant que germe de fibré
vectoriel (sans métrique), c'est-à-dire isomorphes à un objet $(U,E)$ où $E$ est libre comme fibré vectoriel. 
Si le germe $(X,x)$ est bon on a  $\breve {\mathsf F}(X,x)=\mathsf F(X,x)$. 
Enfin nous noterons $\widetilde{\mathsf F}(X,x)$ le groupoïde des fibrés vectoriels sur $\widetilde{(X,x)})$
(on ne considère  donc comme flèches que les isomorphismes de fibrés vectoriels).

On peut récapituler ce qui précède à travers le théorème suivant. 

\begin{theo}\label{theo-fibre-residuel}
Soit $X$ un espace $k$-analytique séparé et soit $x$ un point de $X$. 
Soit $n$ un entier. 

\begin{enumerate}
\item La formule $E\mapsto \widetilde{E_x}$ définit un foncteur de $\mathsf F (X,x)$ vers
$\widetilde{\mathsf F}(X,x)$ qui préserve le rang, commute au produit tensoriel, à la formation du dual, et au changement de base. 

\item Supposons que
le germe $(X,x)$ est presque bon.  Le foncteur $E\mapsto \widetilde{E_x}$
est essentiellement surjectif, et c'est déjà le cas de sa restriction à $\breve{\mathsf F}(X,x)$ ; cette dernière
est également pleine. 

\item Supposons que le germe $(X,x)$ est bon. Le foncteur $E\mapsto \widetilde{E_x}$ est plein.

\end{enumerate}

\end{theo}

\begin{proof}
L'énoncé (a)
résulte de \ref{fibre-res-foncteur}, \ref{fibre-res-tenseur}
et \ref{fibre-res-dual}. 
Supposons que $(X,x)$ est presque bon. La proposition 
\ref{prop-fibreres-essurj} assure précisément 
que la restriction à $\breve{\mathsf F}$ du foncteur 
$E\mapsto \widetilde{E_x}$ 
est essentiellement surjective. Montrons qu'elle est pleine. 
Soit donc $E$ et $F$ deux objets de $\breve{\mathsf F}_n$, et supposons donné
un isomorphisme $\theta$ entre $\widetilde{E_x}$ et $\widetilde{F_x}$ ; soit
$\theta'$ le morphisme induit de $\widetilde{E(x)}$ sur $\widetilde{F(x)}$. 
Le corollaire \ref{coro-releve-isometrie} assure qu'il existe un 
isomorphisme $u$ entre les germes en $x$ de fibrés vectoriels (non métrisés) 
$E$ et $F$ qui induit une isométrie de $E(x)$ sur $F(x)$ dont la réduction 
$\widetilde{u(x)}$  coïncide
avec $\theta'$. 
Puisque $\theta'$induit l'isomorphisme 
$\theta$ entre $\widetilde{E_x}$ et 
$\widetilde{F_x}$, le lemme \ref{lem-isometrie-residuelle} garantit que
$u$ est en réalité une isomorphisme entre germes de fibrés vectoriels
métrisés sur $(X,x)$, 
ce qui montre la plénitude attendue, et achève la preuve de (b). 

L'assertion (c) résulte de (b) et du fait que $\breve{\mathsf F}(X,x)
=\mathsf F(X,x)$ lorsque
$(X,x)$ est bon. 

\end{proof}

\begin{lemm}\label{lemm-max-vai}
Soit $X$ un espace $k$-analytique séparé, soit $x$ un point de $X$, et soit $E$
un fibré vectoriel métrisé paralinéaire entier sur $X$.
Soit $s$ une section de $E$ sur $X$ ne s'annulant pas en $x$ et soit $(a_1,\ldots, a_n)$ une famille
non vide de fonctions analytiques sur $X$
inversibles en $x$. Pour tout $i$, posons $r_i=\abs{a_i(x)}
\cdot \|s(x)\|^{-1}$. 
Les assertions suivantes sont équivalentes : 

\begin{itemize}
\item[(i)] il existe un voisinage analytique $U$ de $x$ dans $X$ tel que
$\|s\|^{-1}=\max_i \abs{a_i}\cdot r_i$ identiquement sur $U$ ; 
\item[(ii)] on a $v(\widetilde{s(x)})^{-1}=\max_i v(\widetilde{a_i(x)})$ pour tout $v\in \widetilde{(X,x)}$.
\end{itemize}
\end{lemm}

\begin{proof}
Pour tout $i$, notons $V_i$ le domaine analytique de $X$ défini par l'égalité
$\|s\|^{-1}=\abs{a_i}\cdot r_i$, et soit $W$ le domaine de $X$ décrit par la conjonction 
d'inégalités $\|s\|^{-1}\leq \abs{a_i}\cdot r_i$ pour $i$ variant de $1$ à $n$. 
Par définition des $r_i$, le point $x$ appartient à $W$
et à chacun des $V_i$. Posons $V=W\cap (\bigcup_i V_i)$. C'est un domaine
analytique de $X$ contenant $x$, et c'est précisément le lieu de validité de l'égalité
$\|s\|^{-1}=\max_i \abs{a_i}\cdot r_i$. 

La réduction $\widetilde{(V,x)}$ est l'intersection de $\widetilde{(W,x)}$ et de la réunion des
$\widetilde{(V_i,x)}$. Par construction, $\widetilde{(W,x)}$ est l'ensemble des 
$v\in \widetilde{(X,x)}$ telles que $v(\widetilde{s(x)})^{-1}
\leq v(\widetilde{a_i(x)})$ pour tout $i$, et $\widetilde{(V_i,x)}$
est l'ensemble des 
$v\in \widetilde{(X,x)}$ telles que $v(\widetilde{s(x)})^{-1}=v(\widetilde{a_i(x)})$. 
Il s'ensuit que $\widetilde{(V,x)}$ est l'ensemble des 
$v\in \widetilde{(X,x)}$ telles que $v(\widetilde{s(x)})=\max_i v(\widetilde{a_i(x)})$. 

L'équivalence requise résulte alors du fait que $\widetilde{(V,x)}=\widetilde{(X,x)}$ si et seulement si 
$V$ est un voisinage de $x$ dans $X$. 
\end{proof}

\begin{coro}\label{coro-max-ai}
Soit $L$ un fibré en droites métrisé paralinéaire entier sur $X$, trivialisable comme fibré
vectoriel au voisinage de $x$.
Soit $s$ une section inversible de $L$ au voisinage de $x$ et soit $(a_i)$ une famille
finie d'éléments de $\mathscr O_{X,x}^\times$. 
Pour tout $i$, posons $r_i=\abs{a_i(x)}\cdot \norm{s(x)}^{-1}$. 

Les assertions suivantes sont équivalentes : 

\begin{enumerate}
\item[(i)] Les $\widetilde{a_is(x)}$ sont des sections
globales de $\widetilde{L_x}$ qui l'engendrent. 
\item[ii)] On a $\|s\|^{-1}=\max_i \abs{a_i}\cdot r_i$ au voisinage de $x$.
\end{enumerate}
\end{coro}
\begin{proof}
Cela découle immédiatement du lemme \ref{lemm-max-vai}
et du lemme \ref{lemm-secglob-max}. 

\end{proof}

\begin{coro}\label{coro-engendre-sections-globales}
Soit $L$ un fibré en droites métrisé paralinéaire entier sur $X$, trivialisable comme fibré
vectoriel au voisinage de $x$.
Soit $s$ une section inversible de $L$ au voisinage de $x$. Supposons le germe
$(X,x)$ presque bon. Les assertions suivantes sont équivalentes : 

\begin{enumerate}
\item[(i)] Le fibré résiduel $\widetilde{L_x}$ est engendré par ses sections globales. 
\item[ii)] Il existe une famille finie 
 finie $(a_i)$ de fonctions appartenant à 
 $\mathscr O_{X,x}^\times$ telles que
 $\|s\|^{-1}=\max_i \abs{a_i}\cdot r_i$ au voisinage de $x$, 
 où $r_i=\abs{a_i(x)}\cdot \norm{s(x)}^{-1}$. 
\end{enumerate}

De plus si elles sont satisfaites et si $\Lambda$ désigne un corps 
gradué de définition de $\widetilde{L_x}$, on peut choisir les $a_i$ de sorte que
$\widetilde{a_i(x)/a_j(x)}\in \Lambda$ pour tout $(i,j)$. 
\end{coro}

\begin{proof}
L'équivalence entre (i) et (ii) résulte du
corllaire \ref{coro-engendre-globalement-gradue}, du lemme
\ref{lemm-max-vai}, et du fait que comme $X$ est presque bon en $x$, 
l'application naturelle de $\mathscr O_{X,x}^\times$ 
vers $\hrt x^\times$ est surjective. 
La dernière affirmation découle de
la remarque \ref{rem-ai-sur-aj} et, là encore, de la surjectivité de 
$\mathscr O_{X,x}^\times\to \hrt x^\times$. 
\end{proof}

\begin{prop}\label{prop.metrique-lisse-pl}
Soit $(X,x)$ un germe d'espace $k$-analytique séparé et presque bon.
Soit $L$ un fibré en droites métrisé paralinéaire entier
sur~$(X,x)$, localement trivialisable
comme fibré en droites. 
La métrique sur $L$ est alors lisse au voisinage de $x$ si 
et seulement si $\widetilde{L_x}$ est de torsion.

\end{prop}
\begin{proof}
Supposons que $\widetilde L_x$ soit de torsion.
Soit $n$ un entier~$\geq 1$ tel que $\widetilde L(x)^{\otimes n}$
soit trivial.
D'après le théorème~\ref{theo-fibre-residuel},
le germe de fibré métrisé~$L^{\otimes n}$ est trivial,
si bien que sa métrique est lisse au voisinage de $x$. 
Par suite, la métrique de $L$ est lisse au voisinage de $x$. 

Réciproquement, supposons que la métrioque de $L$
soit lisse au voisinage de $x$. 
Soit $s$ une section inversible de $L$ au voisinage de $x$. 
La fonction~$\log(\norm s)$ est à la fois paralinéaire entière
et lisse au voisinage de $x$. 
D'après la proposition~\ref{prop.pl-lisse-berk}, il existe un
entier~$n\geq 1$ et une fonction holomorphe inversible~$f$ 
sur $(X,x)$ tels 
que $ n \log(\norm s) = \log(\abs f)$.
La base locale $f^{-1}s^{\otimes n}$ de~$L^{\otimes n}$
est alors de norme constante~$1$.
Le fibré résiduel de~$L^{\otimes n}$ est alors trivial,
si bien que $\widetilde {L_x}^{\otimes n}$ est trivial.
\end{proof}

\section{Fibrés résiduels : le cas strict}
On reprend les notations des sections précédentes: 
$X$ est un espace $k$-analytique séparé, $x$ un point de $X$
et $E$ un fibré vectoriel métrisé paralinéaire entier sur $X$  dont
on note $n$ le rang. 
On suppose de plus que $X$ est strict. 

\begin{lemm}\label{lemm-fibres-strict}
Les assertions suivantes sont équivalentes : 

\begin{itemize}
\item[(i)] La métrique de $E$ est strictement paralinéaire entière au voisinage
de $x$. 

\item[(ii)] Le fibré résiduel $\widetilde{E_x}$ est strict. 
\end{itemize}
\end{lemm}

\begin{proof}
Supposons (i). 
Soit $(V_i)$ une famille finie de domaines strictement $k$-analytiques, contenant $x$ et dont 
la réunion est un voisinage de $x$, et tels que $E_{V_i}$ possède pour tout $i$ une base orthonormée $(e_{ij})_j$. 
Il résulte de la construction de \ref{rappels-fibgradue-metriques} que $(\widetilde{e_{ij}(x)})_j$ est pour
tout $i$ une base de $\widetilde{E_x}|_{\widetilde{(V_i,x)}}$, dont les éléments sont homogènes
de  degré $1$ ; puisque les
$\widetilde{(V_i,x)}$ recouvrent $\widetilde{(X,x)}$, le fibré $\widetilde{E_x}$ est strict. 

Réciproquement, supposons que $\widetilde{E_x}$ est strict. 
Il existe un recouvrement fini $(U_i)$ de $\widetilde{(X,x)}$ par des ouverts quasi-compacts
stricts et, pour tout $i$, une base $(\epsilon_{ij})$ de
$\widetilde{E_x}|_{U_i}$ constituée de sections de degré $1$.
On peut supposer que les $U_i$ sont affines.  Chacun des $U_i$ est alors égal à $\widetilde{(V_i,x)}$ pour
un bon domaine strictement $k$-analytique $V_i$ de $X$ contenant $x$, et la réunion des $V_i$
est un voisinage de $x$. Puisque $V_i$ est bon (ce qui entraîne que $E|_{V_i}$ est engendré par ses sections
globales au voisinage de $x$ dans $V_i$, 
il résulte du lemme 
\ref{lem-exist-sections-locales}
que chaque $\epsilon_{ij}$ se relève en une section $e_{ij}$ de $E$ au voisinage de $x$, et 
l'on déduit du corollaire \ref{coro-BON} que les $e_{ij}$ forment une base orthonormée de $E$ sur un voisinage 
de $x$ dans $V_i$. 
On peut donc restreindre les $V_i$ de sorte que chacun des $E|_{V_i}$ possède une base orthonormée. Par
conséquent, la métrique paralinéaire de $E$ est strictement entière au voisinage de $x$. 
\end{proof}

\subsection{}\label{ss.description-fibre-residuel-strict}
Supposons que les assertions équivalentes ci-dessus soient satisfaites. 
Rappelons que pour tout domaine analytique strict $V$ de $X$  contenant $x$, nous 
désignons par $\widetilde{(V,x)}^1$ l'ouvert quasi-compact de $\zr{\hrt x^1}{\widetilde k^1}$ 
dont il provient (qui est aussi l'image de  $\widetilde{(V,x)}^1$ sur $\zr{\hrt x^1}{\widetilde k^1}$). 
Nous noterons $\widetilde{E_x}^1$ le fibré vectoriel sur $\widetilde{(X,x)}^1$ d'où
$\widetilde{E_x}$ provient. Cette notation est cohérente avec le fait que pour tout
domaine analytique strict $V$ de $X$  contenant $x$, les sections de
$\widetilde{E_x}^1$ sur $\widetilde{(V,x)}^1$ sont précisément les sections homogènes de degré $1$
de $\widetilde{E_x}$ sur $\widetilde{(V,x)}$. 

Il résulte de nos constructions que le fibré résiduel strict $\widetilde{E_x}^1$ peut être directement décrit comme
suit, en termes de $\hrt x^1$-espace vectoriel gradué métrisé. 

Le $\hrt x^1$-espace vectoriel gradué sous-jacent à $\widetilde{E_x}^1$  est $\widetilde{E(x)}^1$, c'est-à-dire
le quotient du groupe des éléments de $E(x)$ de norme $\leq 1$ par son sous-groupe formé des éléments
de norme $<1$.  La métrique sur $\widetilde{E(x)}^1$ définissant le fibré vectoriel $\widetilde{E_x}^1$ est 
caractérisé par la propriété suivante. Soit $V$ un domaine strictement $k$-analytique de $X$ contenant $x$
et soit $s$ une section de $E$ sur $V$ telle que $\norm{s(x)}=1$. Soit $W$ le domaine strictement
$k$-analytique de $X$ défini par l'inégalité $\norm s\leq 1$. Alors $\widetilde{(W,x)}^1$ est l'ensemble
des valuations $v\in \widetilde{(V,x)}^1$ telles que $v(\widetilde{s(x)})\leq 1$.

Chacun des énoncés vus à la section précédente admet une déclinaison stricte, qui peut soit être déduite de l'énoncé en question 
\textit{via} l'équivalence de catégories de \ref{ss-fibre-strict}, soit être démontrée directement en suivant
\textit{mutatis mutandis} la preuve de l'énoncé original. Nous allons indiquer ci-dessous ces différentes variantes
strictes, sans preuve. On suppose à partir de maintenant que 
le fibré métrisé paralinéaire $E$ est strictement entier.

\begin{lemm}[variante stricte du lemme \ref{lem-isometrie-residuelle}]
Soit $F$ un fibré vectoriel métrisé paralinéaire strictement entier 
sur $X$. 
Soit $V$ un domaine strictement analytique de $X$ contenant $x$.
Soit $u$ un isomorphisme de fibrés vectoriels de $E$ sur $F$ tel que
le morphisme induit $u(x)\colon E(x)\to F(x)$ soit 
une isométrie ; soit
$\theta$ l'isomorphisme de 
$\widetilde {E(x)}^1$ sur $\widetilde{F(x)}^1$ induit par $u(x)$. 
Les assertions suivantes sont équivalentes : 

\begin{itemize}
\item[(i)] il existe un voisinage analytique $U$
de $x$ dans $X$ tel que $u|_{U\cap V}$ soit une isométrie ; 

\item[(ii)] l'application $\theta$ est une isométrie au-dessus de $\widetilde{(V,x)}^1$, c'est-à-dire que 
l'on a $v(\theta(s))=v(s)$ pour toute section $s$
de $\widetilde {E(x)}^1$ et toute $v\in \widetilde{(V,x)}^1$ ; 
\item[(iii)] l'application $\theta$ induit un isomorphisme de
$\widetilde{E_x}^1|_{\widetilde{(V,x)}^1}$ sur 
$\widetilde{F_x}^1|_{\widetilde{(V,x)}^1}$. 
\end{itemize}
\end{lemm}

\begin{coro}[variante stricte du corollaire \ref{coro-BON}]
Soit $(e_i)$ une base de $E$ telle que $\norm{e_i(x)}=1$ pour tout $i$
et telle que les $\widetilde{e_i(x)}$ forment une base
de $\widetilde{E(x)}^1$ et soit $V$ un domaine strictement analytique de $X$ contenant $x$. 
Les assertions suivantes sont équivalentes : 

\begin{itemize}
\item[(i)] il existe un voisinage analytique $U$ de $x$ tel que les
$e_i|_{V\cap U}$ constituent une base orthonormée
de $E|_{V\cap U}$ ; 

\item[(ii)] la base $(\widetilde{e_i(x)})$
est orthonormée au-dessus de $\widetilde{(V,x)}^1$ ;

\item[(iii)] les $\widetilde{e_i(x)}|_{\widetilde{(V,x)}^1}$ constituent une base de $\widetilde{E_x}^1
|_{\widetilde{(V,x)}^1}$. 

\end{itemize}
\end{coro}

\begin{coro}[variante stricte du corollaire \ref{coro-releve-isometrie}]
Supposons que $X$ est presque bon en $x$ et que $E$ est trivialisable au voisinage de $x$ comme fibré
vectoriel. Soit $F$ un fibré vectoriel métrisé paralinéaire strictement
entier sur $X$, lui aussi trivialisable comme fibré vectoriel
au voisinage de $X$. Soit $V$ un domaine strictement 
analytique de $X$ contenant $x$ et soit $\theta$ un isomorphisme
de $\widetilde{E(x)}^1$ sur $\widetilde{F(x)}^1$ induisant un isomorphisme
de $\widetilde{E_x}^1|_{\widetilde{(V,x)}^1}$ sur $\widetilde{F_x}^1|_{\widetilde{(V,x)}^1}$. Il existe alors un voisinage
analytique $U$ de $x$ dans $X$ et un isomorphisme de fibrés vectoriels
$u\colon E_U\simeq F|_U$ tel que $u|_{V\cap U}$ soit une 
isométrie et tel que $\widetilde{u(x)}=\theta$. 
\end{coro}

\begin{prop}[variante stricte de la proposition \ref{prop-fibreres-essurj}]
Supposons que $X$ est presque bon en $x$. Soit $M$ un fibré vectoriel
sur $\widetilde {(X,x)}^1$. Il existe un voisinage analytique $U$ de
$x$ dans $X$ et un fibré vectoriel métrisé paralinéaire strictement 
entier $F$ sur $X$, 
trivialisable au voisinage de $x$ comme fibré vectoriel, tel que
$\widetilde{F_x}^1\simeq M$.
\end{prop}

Notons $\mathsf F(X,x)^{\mathrm{str}}$, 
$\widetilde{\mathsf F}(X,x)^{\mathrm{str}}$
et $\breve{\mathsf F}^{\mathrm{str}}(X,x)$ les variantes
strictes des catégories introduites en 
\ref{ss-categories-germes-fibres}. 

\begin{theo}
Soit $X$ un espace strictement 
$k$-analytique séparé et soit $x$ un point de $X$. 
Soit $n$ un entier. 

\begin{enumerate}
\item La formule $E\mapsto \widetilde{E_x}^1$ définit un foncteur de $\mathsf F^{\mathrm{str}} (X,x)$ vers
$\widetilde{\mathsf F}^{\mathrm{str}}(X,x)$ qui préserve le rang, commute au produit tensoriel, à la formation du dual, et au changement de base. 

\item Supposons que
le germe $(X,x)$ est presque bon.  Le foncteur $E\mapsto \widetilde{E_x}^1$
est essentiellement surjectif, et c'est déjà le cas de sa restriction à $\breve{\mathsf F}^{\mathrm{str}}(X,x)$ ; cette dernière
est également pleine. 

\item Supposons que le germe $(X,x)$ est bon. Le foncteur $E\mapsto \widetilde{E_x}^1$ est plein.
\end{enumerate}
\end{theo}

\begin{lemm}[variante stricte du lemme \ref{lemm-max-vai}]
\label{lemm-max-vai-str}
Soit $X$ un espace strictement 
$k$-analytique séparé, soit $x$ un point de $X$, et soit $E$
un fibré vectoriel métrisé paralinéaire strictement entier sur $X$.
Soit $s$ une section de $E$ sur $X$ telle que $\norm{s(x)}=1$
soit $(a_1,\ldots, a_n)$ une famille
non vide de fonctions analytiques sur $X$
telles que $\abs{a_i(x)}=1$ pour tout $i$. 
Les assertions suivantes sont équivalentes : 

\begin{itemize}
\item[(i)] il existe un voisinage analytique $U$ de $x$ dans $X$ tel que
$\|s\|^{-1}=\max_i \abs{a_i}$ identiquement sur $U$ ; 
\item[(ii)] on a $v(\widetilde{s(x)})^{-1}=\max_i v(\widetilde{a_i(x)})$ pour tout $v\in \widetilde{(X,x)}^1$.
\end{itemize}
\end{lemm}

\begin{coro}[variante stricte du corollaire \ref{coro-max-ai}]
\label{coro-max-ai-str}
Soit $L$ un fibré en droites métrisé paralinéaire
strictement entier sur $X$, trivialisable comme fibré
vectoriel au voisinage de $x$.
Soit $s$ une section de $L$ au voisinage de $x$
telle que $\norm{s(x)}=1$ et et soit $(a_i)$ une famille
finie d'éléments de $\mathscr O_{X,x}^\times$ tels que $\abs{a_i(x)}=1$
pour tout  $i$.

Les assertions suivantes sont équivalentes : 

\begin{enumerate}
\item[(i)] Les $\widetilde{a_is(x)}$ sont des sections
globales du fibré en droites $\widetilde{L_x}^1$ qui l'engendrent. 
\item[ii)] On a $\|s\|^{-1}=\max_i \abs{a_i}$ au voisinage de $x$.
\end{enumerate}
\end{coro}

\begin{coro}[variante stricte du corollaire \ref{coro-engendre-sections-globales}]
\label{coro-engendre-sections-globales-str}
Soit $L$ un fibré en droites métrisé paralinéaire
strictement entier sur $X$, trivialisable comme fibré
vectoriel au voisinage de $x$.
Soit $s$ une section de $L$ au voisinage de $x$
telle que $\norm{s(x)}=1$. 
pposons le germe
$(X,x)$ presque bon. Les assertions suivantes sont équivalentes : 

\begin{enumerate}
\item[(i)] Le fibré en droites résiduel $\widetilde{L_x}^1$ 
est engendré par ses sections globales. 
\item[ii)] Il existe une famille finie 
 finie $(a_i)$ de fonctions appartenant à 
 $\mathscr O_{X,x}^\times$ telles que
 $\abs{a_i(x)}=1$ pour tout $i$ et que 
 $\|s\|^{-1}=\max_i \abs{a_i}$ au voisinage de $x$. 
\end{enumerate}

De plus si elles sont satisfaites et si $\Lambda$ désigne un corps 
de définition de $\widetilde{L_x}^1$, on peut choisir les $a_i$ de sorte que
$\widetilde{a_i(x)/a_j(x)}\in \Lambda$ pour tout $(i,j)$. 
\end{coro}

\begin{prop}[variante stricte de la proposition \ref{prop.metrique-lisse-pl}]
\label{prop.metrique-lisse-pl-str}
Soit $(X,x)$ un germe d'espace strictement 
$k$-analytique séparé et presque bon.
Soit $L$ un fibré en droites métrisé paralinéaire
strictement entier sur~$(X,x)$, localement trivialisable
comme fibré en droites. 
La métrique sur $L$ est alors lisse au voisinage de $x$ si 
et seulement si $\widetilde{L_x}^1$ est de torsion.
\end{prop}

\section{Dimension de définition d'un fibré métrisé paralinéaire entier}\label{dimdefn} 


\begin{defi}
Soit $(X,x)$ un germe d'espace $k$-analytique séparé
et soit $(E_i)$ une famille de fibrés vectoriels métrisés
paralinéaires entiers sur~$X$.
La \emph{dimension de définition résiduelle simultanée}\index{dimension de définition résiduelle simultanée d'une famille de fibrés vectoriels métrisés
paralinéaires entiers} des~$E_i$ est le plus petit entier~$d$ tel qu'il existe
un corps gradué de définition de $(\widetilde{(X,x)}, (\widetilde E_{i,x}))$
de degré de transcendance $d$. 
\end{defi}

Ctte dimension de définition résiduelle
est aussi la dimension minimale d'un pré-modèle 
de $(\widetilde{(X,x)},  (\widetilde E_{i,x}))$. 
Elle est inférieure ou égale au degré de transcendance
de~$\hrt x$ sur~$\tilde k$ et en particulier, 
elle est inférieure ou égale à $\dim_x(X)$.
Lorsqu'il y a égalité, 
le point~$x$ est d'Abhyankar:
l'extension~$\hrt x$ de~$\tilde k$ 
est de degré de transcendance~$\dim_x(X)$,
elle est de type fini
et le groupe $\abs{k^\times}$ est d'indice fini dans $\abs{\hr x^\times}$
(\S\ref{lemm.abhyankar-finitude}).

Lorsque $(E_i)$ est une famille de fibrés vectoriels métrisés paralinéaires entiers
sur un espace $k$-analytique presque bon~$X$,
on définit sa dimension de définition résiduelle
simultanée des~$E_i$ en un point~$x$ en la restreignant au germe~$(X,x)$.

\begin{prop}\label{prop.dimdefn}
Soit~$X$ un espace $k$-analytique séparé de dimension $\leq n$. 
Soit~$d$ un entier naturel.
Soit $(E_i)$ une famille finie de fibrés vectoriels métrisés 
paralinéaires entiers sur~$X$.
L'ensemble~$W^d$ des points~$x$ de~$X$
en lesquels la dimension de définition résiduelle simultanée 
des~$E_i$ est supérieure ou égale à~$d$ est contenu dans une partie
faiblement analytique de~$X$, de dimension~$\leq n-d$.
En particulier, si $d=n$, cet ensemble est fermé et discret.
\end{prop}
\begin{proof}
Soit $x\in X$.  
Soit $(X_j)_{j \in J}$ une famille finie de domaines affinoïdes
de $X$ contenant $x$, dont la réunion est un voisinage de~$x$, 
et telle que pour tout $(i,j)$ le fibré métrisé
$E_i|_{X_j}$ 
possède une base orthogonale $(e_{ij\alpha})_\alpha$ 
dont les vecteurs sont de norme constante. Pour toute partie
non vide $P$ de $J$, notons $X_P$ 
l'intersection des $X_j$ pour $j$ dans $P$, et $\rho_P\colon 
X_P\to \widetilde{X_P}$ l'application de spécialisation. 
Soit $U$ l'intérieur topologique (dans $X$) de la réunion des 
$X_i$. Il suffit de démontrer que $W^d\cap U$ est faiblement 
analytique de dimension $\leq n-d$. 

Soit $y$ un point de $U$ et soit $J(y)$ l'ensemble des indices $j$
tels que $y\in X_j$. Posons $A=\mathscr O(X_{J(y)})$, 
$M_i=E_i(X_{J(y)})$, et $\eta=\rho_{J(y)}(y)$ ; soit
$\kappa(\eta)$ le corps résiduel gradué de $\eta$. 

L'évaluation en $y$ définit un morphisme contractant de
$A$ vers $\hr y$, qui induit un morphisme
de $\widetilde A$ vers $\hrt y$ qui par définition se factorise
par $\kappa(\eta)$. 

Pour tout $i$, la semi-norme spectrale
$s\mapsto \sup_{z\in X_{J(y)}}\norm{s(x)}$
fait de $M_i$ 
un $A$-module semi-normé, et l'évaluation définit 
un morphisme contractant de $A$-modules normés de 
$M_i$ vers $E(y)$, d'où un morphisme de 
$\widetilde A$-modules gradués de 
$\widetilde {M_i}$ vers $\widetilde {E_i(y)}$,
qui se factorise par $N_i:=\widetilde {M_i}\otimes_{\widetilde A}
\kappa(\eta)$. Choisissons $j\in J(y)$. 

Quel que soit l'indice $i$, on a l'égalité 
$\norm{\sum \lambda_\alpha e_{ij\alpha}}=\sup_\alpha \norm{\lambda_\alpha}\cdot \norm{e_{ij\alpha}}$
pour toute famille $(\lambda_\alpha)$ d'éléments de $A$ ; il en résulte que les
$\widetilde{e_{ij\alpha}}$ forment une base du $\widetilde{A}$-module
$\widetilde{M_i}$, puis que leurs images
dans 
$N_i$ (que nous noterons 
$\widetilde{e_{ij\alpha}}(\eta)$) forment une base de ce dernier
comme $\kappa(\eta)$-espace
vectoriel gradué. Et comme les $e_{ij\alpha}(y)$ forment une base orthogonale de $E_i(y)$, 
les $\widetilde{e_{ij\alpha}(y)}$ forment une base de $\widetilde{E_i(y)}$. 
Le morphisme naturel de $N_i\otimes_{\kappa(\eta)}
\hrt y$ dans 
$\widetilde{E_i(y)}$ envoyant $1\otimes  \widetilde{e_{ij\alpha}}(\eta)$
sur $\widetilde{e_{ij\alpha}(y)}$ pour tout $\alpha$, c'est un isomorphisme. 

Dans ce qui suit, nous nous permettrons de considérer $\kappa(\eta)$
comme un sous-corps gradué de $\hrt y$, et chacun des $N_i$ comme un
sous-$\kappa(\eta)$-espace vectoriel de $\widetilde {M_i}$. 
Avec ces conventions, chacun des  $\widetilde{e_{ij\alpha}(y)}$
appartient à $N_i$. 

Comme $y$ appartient à l'intérieur
de $\bigcup _{j\in J}X_j$, la
réunion des $X_j$ pour $j\in J(y)$ est un voisinage de $y$ dans $X$, 
si bien que $\widetilde{(X,y)}=\bigcup_{j\in J(y)}\widetilde
{(X_j,y)}$. 
Pour tout $j\in J(y)$, posons $B_j=\mathscr O(X_j)$. L'ouvert 
$\widetilde
{(X_j,y)}$ est égal à $\zr{\hrt y}{\widetilde k}\{\mathsf B_j\}$, 
où $\mathsf B_j$ désigne l'image de $\widetilde {B_j}$ dans 
$\hrt y$, et 
$\{\widetilde{e_{ij\alpha}(y)}\}_\alpha$ est une base
de $\widetilde{E_{i,y}}|_{\widetilde {(X_j,y)}}$. 
De plus, le morphisme de $B_j$ vers $\hr y$ se factorise par $A$ ; 
en passant aux réductions graduées, on en déduit que le morphisme 
de $\widetilde {B_j}$ vers $\hrt y$ se factorise par $\widetilde A$, 
puis par $\kappa(\eta)$ ; par conséquent, $\mathsf B_j\subset \kappa(\eta)$
pour tout $j$. Joint au fait que les $\{\widetilde{e_{ij\alpha}(y)}\}_\alpha$
appartiennent à $N_i$ pour tout $(i,j,\alpha)$ ceci entraîne au vu du lemme
\ref{lemm-critere-def-fibre} que $\widetilde{(X,x)}$ est l'image réciproque 
d'un ouvert $V$ de $\zr{\kappa(\eta)}{\widetilde k}$ par la projection canonique
$\pi \colon \zr{\hrt y}{\widetilde k}$ et que chacun des $\widetilde{E_{i,y}}$
est isomorphe à $\pi^*F_i$ pour un certain fibré vectoriel $F_i$
sur $U$. Par conséquent, la dimension de définition résiduelle simultanée des 
$E_i$ en $y$ est majorée par le degré de transcendance de 
$\kappa(\eta)$ sur $\widetilde k$. 

Il résulte de ce qui précède que $W_d\cap U$ est contenu dans
\[ \bigcup_{\emptyset
\subset P\subset J}
\{z\in X_P, \mathrm{deg.\,tr.\,}(\kappa(\rho_P(z))/\widetilde k)
\geq d\}.\]
En vertu de la proposition \ref{prop.tube-faiblement-analytique}, 
l'ensemble ci-dessus est une partie faiblement analytique de dimension 
majorée par $n-d$. Il en va dès lors de même de son intersection avec $U$, 
ce qui achève la démonstration. 

\end{proof}

\begin{defi}\label{def-defanal-simultanee}
Soit $(X,x)$ un germe presque bon d'espace $k$-analytique.
Soit $(E_i)$ une famille de fibrés vectoriels métrisés 
paralinéaires entiers
sur~$(X,x)$ dont les fibrés vectoriels sous-jacents sont trivialisables.
La dimension de définition analytique simultanée des~$E_i$
est le plus petit entier~$d$ tel qu'il existe
un germe presque bon~$(Y,y)$ de dimension~$d$, 
une famille $(F_i)$ de fibrés vectoriels métrisés paralinéaires entiers~$(Y,y)$ 
dont les fibrés sous-jacents sont trivialisables 
et un morphisme plat et sans bord 
$f\colon (X,x)\to (Y,y)$ 
tel que $E_i\simeq f^*F_i$ pour tout~$i$.  
\end{defi}

Cette dimension  est bien sûr inférieure ou égale à~$\dim_x(X)$,
et supérieure  ou égale à celle de définition résiduelle simultanée.

\begin{prop}\label{prop.dim-def-analytique}
Soit~$(X,x)$ un germe presque bon, 
soit~$(E_i)$ une famille finie de fibrés vectoriels métrisés paralinéaires entiers
sur~$(X,x)$ dont les fibrés sous-jacents sont trivialisables,
et soit~$d$ la dimension de définition résiduelle simultanée des~$E_i$.

\begin{enumerate}
\item 
Supposons qu'il existe un sous-corps gradué $L$ de $\hrt x$, 
contenant $\widetilde k$, de type fini et de degré de transcendance $d$ sur $\widetilde k$,
tel que le module des différentielles graduées $\Omega_{L/\widetilde k}$ soit de dimension $d$
comme $L$-espace vectoriel, et tel que $L$ soit un corps gradué de définition
de $(\widetilde{(X,x)}, (\widetilde{E_{i,x}})_i)$. 
La dimension de définition analytique  simultanée
de la famille~$(E_i)$ est alors égale à~$d$. Plus précisément, on peut trouver
$(Y,y)$ comme dans la définition \ref{def-defanal-simultanee}
tel que $\hrt y$ soit égal à $L$ lorsqu'on le voit comme sous-corps gradué  de
$\hrt x$ via $f\colon (X,x)\to (Y,y)$. 

\item Si le corps~$\widetilde k$ est parfait, 
la dimension de définition analytique  simultanée
de la famille~$(E_i)$ est égale à~$d$.

\item Si les $E_i$ sont des fibrés en droites
et si la famille $(E_i)$ est finie, il existe
un entier naturel~$q$, puissance 
de l'exposant caractéristique~$p$ de~$\widetilde k$,
tel que la famille~$(E_i^{\otimes p^N})$ soit de dimension de définition
analytique simultanée inférieure ou égale à~$d$.
\end{enumerate}
\end{prop}
Pour les notions relatives au module des différentielles ou à la notion de corps
parfait et d'extension séparable dans le contexte gradué, on pourra se reporter 
à la première section de \cite{ducros2013}. Précisons juste que $\widetilde k$ est toujours parfait en 
caractéristique résiduelle nulle ; en caractéristique résiduelle $p>0$ il l'est si et seulement si 
$\widetilde k^p=\widetilde k$, ce qui revient à demander que le corps $\widetilde k_1$ soit parfait 
et que $\abs{\widetilde k^\times}=\abs{k^\times}$ soit $p$-divisible.

\begin{proof}
Montrons tout d'abord l'énoncé~1). 
Soit $(f_1,\dots,f_d)$ une famille de fonctions analytiques
sur $(X,x)$ inversibles en $x$ et telles que
les $\widetilde{f_i(x)}$ appartiennent à $L$ et que leurs
différentielles forment une base 
de $\Omega_{L/k}$ ; leur existence résultent
des hypothèses faites sur $L$ et du caractère presque bon de $(X,x)$. 
Les $\widetilde{f_i(x)}$ forment alors
une base de transcendance de $L$ sur $\widetilde k$, et $L$
est finie séparable
sur son sous-corps engendré par $\widetilde k$ et les 
$\widetilde{f_i(x)}$. 

Posons $r=(\abs{f_1(x)}, \ldots, \abs{f_d(x)})$ et $f=(f_1,\dots,f_d)$; 
par construction, $f$ est un morphisme de~$(X,x)$
vers $(\gm^d,\eta_r)$, \emph{via}
lequel $\hrt{\eta_d}$ s'identifie à $\widetilde k(\widetilde {f_1(x)},
\dots,\widetilde {f_d(x)})$. 

Soit $\Lambda$ l'extension modérément
ramifiée de $\hr{\eta_r}$
dont le corps résiduel gradué est l'extension
finie séparable $L$ de $\hrt{\eta_d}$
(voir la section 2 \cite{ducros2013} pour une présentation de la théorie
de la ramification modérée en termes de réductions graduées). 
Par la proposition~3.4.1 de~\cite{berkovich1993}, 
l'extension $\hr {\eta_r}\hookrightarrow \Lambda$ est induite par un morphisme
étale de germes $(Y,y)\to (\gm^d,\eta_r)$.
Puisque $(\gm^d,\eta_r)$ est sans bord, $(Y,y)$
est sans bord, ce qui veut
dire que $\widetilde{(Y,y)}
=\zr{\hrt y}{\widetilde k}$. 
Le produit tensoriel
$\hr x\otimes_{\hr {\eta_r}}\hr y$ est une $\hr x$-algèbre étale modérément ramifiée, 
d'algèbre étale résiduelle graduée
$\hrt x\otimes_{\widetilde k(\widetilde {f_1(x)},
\ldots,\widetilde {f_d(x)}}L$, et le plongement de $L$ dans~$\hrt x$
fournit un quotient de cette algèbre étale égal à~$\hrt x$ ; 
ce dernier correspond
à un quotient de $\hr x\otimes_{\hr {\eta_d}}\hr y$ isomorphe à $\hr x$, 
c'est-à-dire en vertu
une fois encore de la proposition 3.4.1 de \cite{berkovich1993}, 
à une section du morphisme
étale $(Y,y)\times_{(\gm^d,\eta_r)} (X,x)\to (X,x)$ ; cette section  induit alors une factorisation de~$f$ par un morphisme $g\colon (X,x)\to (Y,y)$. Par construction, le plongement 
$\hrt y\hookrightarrow \hrt x$ induit par~$g$ s'identifie au plongement initialement donné
de~$L$ dans~$\hrt x$. Nous n'utiliserons désormais
plus la notation $L$, mais uniquement $\hrt y$. 
L'image
de $\widetilde{(X,x)}$ sur 
$\zr{\hrt y}{\widetilde k}$ est un ouvert quasi-compact, 
donc est de la forme $\widetilde{(Y',y)}$ pour un certain
domaine analytique $Y'$ de $Y$ contenant $y$. En tant que domaine
d'un bon germe, $(Y',y)$ est presque bon. 
Et puisque ${\hrt y}$ est un corps gradué de définition de
$(\widetilde{(X,x)}, (\widetilde{E_{ix}})_i)$, 
on a l'égalité
$\zr{\hrt x}{\widetilde k}\times_{\zr {\hrt y}
{\widetilde k}}\widetilde{(Y',y)}=
\widetilde{(X,x)}$, et chacun des fibrés
$\widetilde{E_{ix}}$ est l'image réciproque
d'un fibré vectoriel $G_i$ sur $\widetilde{(Y',y)}$. 
Le morphisme 
$(X,x)\to (Y,y)$ se factorise dès lors par un morphisme
sans bord $h\colon (X,x)\to (Y',y)$. 
Le germe 
$(Y',y)$ étant quasi-étale sur $(\gm^ d,\eta_1)$ il est quasi-lisse 
et en particulier réduit, et $y$ en est un point d'Abhyankar. 
Ceci entraîne que $h$ est plat, 
par exemple en invoquant le théorème 10.3.7 de~\cite{ducros2018}.

Pour tout~$i$, le fibré~$G_i$ sur~$\widetilde{(Y',y)}$
se relève d'après 
d'après le théorème \ref{theo-fibre-residuel}
en un fibré métrisé
paralinéaire entier $F_i$ sur~$(Y',y)$ 
dont le fibré sous-jacent est trivialisable.
Par fonctorialité, chacun des fibrés métrisés paralinéaires
entiers $h^*F_i$
sur~$(X,x)$ admet $\widetilde{E_{ix}}$ comme fibré résiduel. 
En utilisant une fois encore
le théorème~\ref{theo-fibre-residuel}, 
on voit que $h^*F_i\simeq E_i$, ce qui achève de montrer l'assertion~(a). 

Lorsque le corps~$\widetilde k$ est parfait, 
la dimension du $L$-espace
$\Omega_{L/\widetilde k}$ espace vectoriel gradué
est égale au degré de transcendance de $L$ sur $\widetilde k$. 
En effet, il suffit pour le voir de montrer que c'est 
le cas lorsqu'on remplace $\widetilde
k$ par $\widetilde k(T_i/r_i)_i[T_i^{1/p^\infty}]_i$ et $L$
par $L(Ti/r_i)_i[T_i^{1/p^\infty}]$ où $(r_i)$ est un polyrayon donné
et où $p$ est l'exposant caractéristrique de $\widetilde k$ --
observons que comme $\widetilde k$ est parfait,
$\widetilde k(T_i/r_i)_i[T_i^{1/p^\infty}]_i$ l'est encore. 
On peut donc, en choisissant convenablement la famille $(r_i)$, 
se ramener au cas où $\abs{L^\times}=\abs{\widetilde k^\times}$ ; 
dans ce cas $L$ est égal à $L_1\otimes_{\widetilde k_1}\widetilde k$, et
le degré de transcendance de $L$ sur $\widetilde k$ est égal
à celui de $L_1$ sur $\widetilde k_1$. Et comme $\widetilde k$
est parfait, $\widetilde k_1$ est parfait, et notre assertion résulte
alors de son analogue non gradué bien connu. 
On en déduit au vu de (a) que la dimension de
définition analytique simultanée des $E_i$ est au plus $d$, 
et comme elle est par ailleurs majorée par $d$, elle
vaut exactement $d$.

Supposons enfin que les~$E_i$ sont de rang~$1$. 
Comme la dimension de définition résiduelle simultanée des~$E_i$ est~$d$
et comme la famille $(E_i)$ est finie, 
la remarque \ref{rem-corps-gradue-def} entraîne
l'existence d'un corps gradué de
définition $L$
de $(\widetilde{(X,x)},(\widetilde{E_{i,x}})_i)$
de type fini et de degré de
transcendance $d$ sur $\widetilde k$. 

Fixons une base de transcendance $(t_1,\ldots, t_d)$ de~$L$ 
sur~$\widetilde k$
et notons~$\Lambda$ la clôture séparable de~$\widetilde k(t_1,\dots,t_d)$ dans~$L$.
Il existe un entier~$q$, puissance
de l'exposant caractéristique de~$\widetilde k$,
tel que~$t^{q}\in\Lambda $ pour tout 
$t\in L$. 
Il résulte alors du lemme \ref{lemme-puissance-p}
que $\Lambda$ est un corps gradué de définition
de $(\widetilde{(X,x)},(\widetilde{E_{i,x}^{\otimes q}})_i)$. 
En vertu de~(1) il s'ensuit que la dimension de définition
analytique simultanée de la famille~$(E_i^{\otimes q})$ 
est au plus égale à~$d$. 
\end{proof}

\subsection{Variantes strictes} 
Si $(X,x)$ est séparé et strict, et si $(E_i)$
est une famille de fibrés
vectoriels paralinéaires strictement entiers, on définit la dimension 
de définition résiduelle simultanée stricte des $E_i$ en $x$ 
comme le degré de transcendance minimal sur
$\widetilde k^1$ d'un corps de définition de $(\widetilde{(X,x)}^1, 
((\widetilde{E_{i,x}})^1)_i$. La définition \ref{def-defanal-simultanee} 
et les propositions \ref{prop.dimdefn}, 
et \ref{prop.dim-def-analytique} ainsi que leurs preuves 
se déclinent alors 
dans ce contexte : il suffit de modifier les 
définitions, énoncés et
démonstrations originels
en imposant partout que les espaces soient stricts et les fibrés paralinéaires strictement entiers, et l'en remplaçant
toutes les réductions graduées (de germes, de fibrés...) par les
réductions strictes correspondantes. 
Notons que si la famille $(E_i)$ supposée finie,
l'avatar strict de l'assertion (a) de la \ref{prop.dim-def-analytique}
peut se reformuler en demandant que
$(\widetilde{(X,x)}^1,((\widetilde{E_{i,x}})^1)_i)$ 
admette un pré-modèle $(V, (F_i))$ où $V$
est une $\widetilde k^1$-variété intègre et génériquement lisse.

\section{Métriques formelles}\label{s-metriques-formelles}

\subsection{}
Soit $E$ un fibré vectoriel de rang~$n$ sur un espace $k$-analytique~$X$.
Considérons un \emph{modèle formel}~$(\mathfrak X,\mathfrak E)$ de~$(X,E)$:
cela signifie que $\mathfrak X$ est $k^\circ$-schéma
formel, plat et localement topologiquement de type fini
muni d'un isomorphisme $\mathfrak X_\eta\simeq X$
et que $\mathfrak E$ est un fibré vectoriel sur~$\mathfrak X$ 
muni d'un isomorphisme   $\mathfrak E_\eta\simeq E$.

Nous ne supposons pas que la valeur absolue de~$k$ est non triviale;
au contraire, nous appliquerons ces considérations
sur le corps résiduel (trivialement valué) de~$k$.
Le lecteur trouvera dans l'article~\cite{thuillier2007}
les définitions essentielles concernant les schémas formels dans le
cas trivialement valué.

\subsection{}
On associe au modèle formel~$(\mathfrak X,\mathfrak E)$
une métrique sur~$E$ de la façon suivante.

Soit $(\mathfrak U_i)$ un recouvrement formel de~$\mathfrak X$
qui trivialise~$\mathfrak E$ et soit, pour tout~$i$, 
une trivialisation~$(e_1,\dots,e_n)$ de~$\mathfrak E|_{\mathfrak U_i}$.
Pour tout~$i$, $\mathfrak U_{i,\eta}$ est un domaine analytique
de~$X$.

Soit $Y$ un espace analytique, $\phi\colon Y\to X$ un morphisme
et $s$ une section de~$\phi^* E$; supposons qu'il existe
un ouvert~$\mathfrak U$ de~$\mathfrak X$ trivialisant~$\mathfrak E$
tel que $\phi(Y)$ soit contenu dans~$\mathfrak U_{\eta}$.

Soit $(e_1,\dots,e_n)$ une base de~$\mathfrak E$ sur~$\mathfrak U$.
Il existe une unique famille $(f_1,\dots,f_n)$
de fonctions holomorphes sur~$Y$ 
telles que $s=\sum f_i \phi^* e_i$; posons alors
$\norm s = \max(\abs{f_1},\dots,\abs{f_n})$;
c'est une fonction continue. 

Soit $\mathfrak U'$ un ouvert de~$\mathfrak X$ trivialisant~$\mathfrak E$
contenant~$\phi(Y)$; choisissons une base~$(e'_1,\dots,e'_n)$
de~$\mathfrak E$ sur~$\mathfrak U'$ et définissons $\norm s'$
de façon analogue.

Prouvons que l'on a $\norm s=\norm s'$.
il existe une matrice inversible $A\in\GL_n(\mathfrak U\cap \mathfrak U')$
telles que $(e_1,\dots,e_n)=(e'_1,\dots,e'_n)\cdot A$.
Si $s=\sum f_i \phi^* e_i = \sum f'_i \phi^* e'_i$, cela entraîne
la relation $(f'_1,\dots,f'_n)^{\mathrm t}=\phi^* A\cdot (f_1,\dots,f_n)^{\mathrm t}$.
Pour tout point~$x\in\mathfrak U\cap\mathfrak U'$, les coefficients
de~$A(x)$ et de~$A^{-1}(x)$ sont de normes~$\leq 1$, 
si bien que 
l'on a $\max(\abs{f_i(y)})=\max(\abs{f'_i(y)})$
pour tout $y\in Y$. Cela démontre que $\norm s=\norm s'$.

D'après la remarque~\ref{rema.metrique.G-local}, on a ainsi défini une métrique sur~$E$.

De telles métriques seront appelées \emph{formelles}\index{métrique!--- formelle}. Elles sont par construction
paralinéaires strictement entières
toute base de $\mathfrak E$ fournissant
une base orthonormée.

\begin{exem}
Soit $n$ un entier~$\geq 0$.
L'exemple le plus important de métrique formelle est fourni par le fibré
en droites~$\mathscr O(1)$, sur
l'espace projectif~$\P^n_k$, et, plus généralement, par ses
puissances tensorielles $\mathscr O(m)$, pour $m\in\Z$.

L'espace projectif relatif 
$\Proj k^\circ[T_0,\dots,T_n]$ sur~$k^\circ$
et le fibré~$\mathscr O(m)$ en fournissent un modèle formel après complétion.
On dira que la métrique formelle associée est
la métrique formelle canonique sur~$\mathscr O(m)$,
ou la métrique de Weil.

Donnons-en une description explicite.
Pour $i\in\{0,\dots,n\}$, soit $\mathfrak U_i$
l'ouvert affine défini par la condition $T_i\neq 0$.
Sa fibre générique est le domaine analytique de~$(\P^n_k)^\an$
formé des points dont les coordonnées homogènes
$(x_0:\dots:x_n)$ vérifient $\abs{x_i}=\max(\abs{x_0},\dots,\abs{x_n})$.
Si $P\in k[T_0,\dots,T_n]$ est un polynôme homogène 
de degré~$m$, notons $s_P$ la section globale correspondante
de~$\mathscr O(m)$ sur~$\mathbf P^n_k$.
Pour tout $i\in\{0,\dots,n\}$, la section~$s_{T_i}$
provient d'une trivialisation de~$\mathscr O(1)$ sur~$\mathfrak U_i$. On a
donc $\norm{s_{T_i}}(x)=1$ pour tout $x\in\mathfrak U_{i,\eta}$.
Cela entraîne la formule familière
\[ \norm {s_P}(x) = \frac{\abs{P(x_0,\dots,x_n)}}{\max(\abs{x_0},\dots,\abs{x_n})^m} \]
pour tout point $x\in(\mathbf P^n_k)^\an$ de coordonnées homogènes
$(x_0:\dots:x_n)$ et tout polynôme $P\in k[T_0,\dots,T_n]$
qui est homogène de degré~$m$.
\end{exem}

\begin{exem}
Considérons deux fibrés vectoriels $\mathfrak E$   et $\mathfrak E'$
sur $\mathfrak X$, et notons $E$ et~$E'$ leurs fibres génériques.
Alors $\mathfrak E\otimes\mathfrak E'$ est un modèle formel de~$E\otimes E'$
et la métrique formelle correspondante se déduit par produit tensoriel
des métriques
formelles sur~$E$ et sur~$E'$.
\end{exem}


%
%
%

\begin{prop}\label{prop.exist.metrique.formelle}
Soit $k$ un corps complet pour une valuation non triviale,
soit~$X$ un espace strictement $k$-analytique 
paracompact et soit $E$ un fibré vectoriel sur~$X$.

\begin{enumerate}
\item Le fibré vectoriel~$E$ possède des métriques
paralinéaires strictement entières. 
\item Une métrique sur~$E$ est
formelle si et seulement si elle est 
paralinéaire strictement entière. 
\end{enumerate}
\end{prop}

Dans le cas des fibrés en droites,
cette proposition avait été démontrée par Gubler
(prop.~7.5, \cite{gubler1998})
sous l'hypothèse supplémentaire que le corps~$k$ est algébriquement clos 
et que $X$ est réduit, quasi-compact et quasi-séparé,
puis dans~\cite{gubler-martin2016}   (prop.~2.10) sans hypothèse.
La  preuve ci-dessous s'inspire de leurs arguments.

\begin{proof}
La première assertion résulte de ce que $(X,E)$ possède un modèle
formel~$(\mathfrak X,\mathfrak E)$, où $\mathfrak E$ est localement libre
(corollaire~\ref{coro.raynaud} et théorème~\ref{theo.app-faisceau}).
En effet, la métrique formelle associée est paralinéaire strictement entière. 

Démontrons  la seconde. Puisque toute métrique formelle est paralinéaire
strictement entière il suffit de prouver qu'une métrique
paralinéaire strictement entière sur~$E$ est formelle.
Soit $(X_i)$ un G-recouvrement de~$X$ tel que 
pour tout~$i$, le fibré métrisé $E|_{X_i}$ possède une base
orthonormée; comme $X$ est paracompact, on peut supposer ce recouvrement
localement fini.
Soit $(g_{ij})$ le $\GL^\circ_n$-cocycle associé à ces données, où
$\GL^\circ_n$ est le G-faisceau qui envoie un domaine $V$
sur le sous-ensemble de $\mathrm M_n(\mathscr O(V))$ formé
des matrices inversibles $M$ telles que tous les coefficients de
$M$ et de $M^{-1}$ soient partout de norme $\leq 1$ ; en d'autres termes, 
c'est le faisceau des automorphismes isométriques du fibré
$\mathscr O_X^n$ muni de la métrique pour laquelle la base canonique est 
orthonormée.

Pour tout couple~$(i,j)$, choisissons un modèle~$\mathfrak X_{ij}$
de~$X_i\cap X_j$ tel que les coefficients de la matrice~$g_{ij}$
et de son inverse appartiennent à $\mathscr O(\mathfrak X_{ij})$
(corollaire~\ref{coro.modele-fonctions}).

Appliquons le théorème~\ref{theo.raynaud} 
à l'espace analytique~$X$, au G-recouvrement localement fini
$(X_{ij})$ et à la famille $(\mathfrak X_{ij})$.
On obtient ainsi un modèle~$\mathfrak X'$ de~$X$
et un recouvrement ouvert $(\mathfrak X'_{ij})$ de~$\mathfrak X$,
tel que $\mathfrak X'_{ij}$ domine~$\mathfrak X_{ij}$ pour tout couple~$(i,j)$.

Pour tout~$i$, posons $\mathfrak X'_i=\mathfrak X'_{ii}$.
Par construction, le cocycle $(g_{ij})$ fournit un
$\GL_n$-cocycle sur~$\mathfrak X'$ subordonné au 
recouvrement $(\mathfrak X'_i)$. Ce cocycle définit un faisceau
localement libre~$\mathfrak E'$ de rang~$n$ sur~$\mathfrak X'$.
Comme sa fibre générique est encore définie par $(g_{ij})$, 
le fibré formé
$\mathfrak E$ est
un modèle de~$E$ et la métrique formelle correspondante sur~$E$
coïncide avec la métrique paralinéaire 
strictement entière donnée initialement.
\end{proof}

\begin{coro}\label{coro.densite-metrique.formelle}
Soit $k$ un corps complet pour une valuation non triviale,
soit~$X$ un espace strictement $k$-analytique 
paracompact et soit $L$ un fibré en droites sur~$X$.
Si $X$ est localement holomorphiquement séparé,
toute métrique continue sur~$L$ peut être approchée
par une métrique paralinéaire strictement entière, pour la
topologie 
de la convergence compacte.
\end{coro}
\begin{proof}
Soit $\norm\cdot$ une métrique continue sur~$L$, soit $K$
une partie compacte de~$X$ et soit $\eps>0$.
Il faut démontrer qu'il existe une métrique strictement~\pl $\norm\cdot'$
sur~$L$ telle que $\norm s(x) e^{-\eps}\leq \norm{s}'(x) \leq \norm s(x) e^{\eps}$ pour toute section G-locale~$s$ de~$L$ et tout $x\in K$.
D'après la proposition~\ref{prop.exist.metrique.formelle},
il existe une métrique 
paralinéaire strictement entière $\norm\cdot '$ sur~$L$.
Soit $f$ l'unique fonction continue sur~$X$
telle que $\norm \cdot' = \norm\cdot \cdot e^{f}$.
D'après le corollaire~\ref{coro.sw-pl},
il existe une fonction strictement~\pl
$u$ sur~$X$ telle que $\abs{u(x)-f(x)}\leq\eps$ pour tout $x\in K$.
La métrique $\norm\cdot' e^{u}$ sur~$L$ 
répond alors à la condition de l'énoncé.
\end{proof}

\subsection{}
Pour $*\in\{\mathscr C^0, \pl, \text{lisse}\}$,
notons $\hPic(X,*)$ l'ensemble des classes d'isométrie
de fibrés en droites sur~$X$ munis d'une métrique~$*$.
Le produit tensoriel le munit d'une structure de groupe abélien.
De plus, l'oubli de la métrique définit un homomorphisme
de groupes $\hPic(X,*)\ra\Pic(X)$. 

Son noyau est le groupe des métriques~$*$ sur le fibré
en droites trivial, c'est-à-dire le groupe des fonctions~$*$
sur~$X$ modulo le sous-groupe des fonctions
(qui sont à la fois continues, lisses et~\pl) de la forme $\log\abs u$,
pour $u\in\Gamma(X,\mathscr O_X^*)$.

Supposons $X$ paracompact.
Lorsque $*=\mathscr C^0$ ou $\text{lisse}$, 
cet homomorphisme est surjectif dès que $X$ est bon
(proposition~\ref{prop.existence-metrique-lisse}).
Il est aussi surjectif 
lorsque $*=\pl$, $\abs{k^\times}\neq 1$ et $X$ est strict.

\subsection{}
Soit $E$ un fibré vectoriel métrisé sur~$X$.
Soit $f\colon Y\ra X$ un morphisme d'espaces $k$-analytiques.
On dispose alors d'un carré cartésien
\[   \begin{tikzcd}
  \mathbf V(f^*E) \ar{r}{f'}\ar{d} & \mathbf V(E)\ar{d} \\
   Y \ar{r}{f} & X \end{tikzcd}.\]
Si $E$ est libre, identifié à~$\mathscr O_X^n$, $f^*E$ est
identifié à~$\mathscr O_Y^n$,
le morphisme~$f'$ s'identifie
à l'application $(\id,f)$ de $\mathbf V(f^*E)=\mathbf A^n\times Y$
dans $\mathbf V(E)=\mathbf A^n\times X$.
On définit alors une métrique
sur le fibré vectoriel~$f^*E$ par la composition
$\mathbf V(f^*E)\xrightarrow {f'}\mathbf V(E)\xrightarrow{\norm\cdot}\R_+$.

Si la métrique de~$E$ est lisse, resp. ~$\pl$,
il en est de même de la métrique de~$f^*E$ ainsi définie.
Par passage aux classes d'isométrie de fibrés en droites métrisés~\pl,
on en déduit un homomorphisme de groupes
 $f^*\colon\hPic(X,*)\ra\hPic(Y,*)$.

\subsection{Fibré résiduel d'un fibré formel}\label{sss.fibre-residuel-formel}
Soit $\mathfrak X$ un $k^\circ$-schéma formel séparé, 
topologiquement localement de type fini et quasi-paracompact
et soit $\mathfrak E$ un fibré vectoriel sur~$\mathfrak X$,
partout de rang~$n$.
Soit $X$ la fibre générique de~$\mathfrak X$ et 
soit $\pi\colon X\to\mathfrak X_{\mathrm s}$ le morphisme de spécialisation.
Comme $\mathfrak X$ est séparé, $X$ est séparé. 
Soit $E$ le fibré vectoriel sur~$X$
fibre générique de~$\mathfrak E$, muni de la métrique formelle 
correspondante.
Soit $x\in X$ et soit~$V$ la $\tilde k$-variété 
$\overline{\{\pi(x)\}}$ (munie de sa structure réduite) ; le 
lemme~\ref{redtemk} assure que $V$ est de type fini
et que $\widetilde{(X,x)}^1=\zr {\hrt x^1}{\tilde k^1}\{V\}$. 
Par ailleurs, le fibré métrisé $E$ est paralinéaire strictement entier. 
Il possède donc un fibré résiduel strict $\widetilde{E_x}^1$ 
en $x$. 

\begin{lemm}\label{lemme-premodele-formel}
Le couple $(V,\mathfrak E)$ est un pré-modèle de
$(\widetilde{(X,x)}^1, \widetilde {E_x}^1$). 
\end{lemm}

\begin{proof}
Choisissons un recouvrement ouvert localement fini 
$(\mathfrak X_i)$
de $\mathfrak X$ tel que $\mathfrak E|_{\mathfrak X_i}$ 
possède pour tout $i$ une base $(e_{ij})_j$ ; désignons
pour tout $i$ par $X_i$ la fibre générique de $\mathfrak X_i$, 
et notons encore $e_{ij}$ l'image de $e_{ij}$ dans $\mathscr O(X_i)$. 
Pour $i$ fixé, les $e_{ij}$ forment par construction
une base orthonormée du fibré métrisé 
$E|_{X_i}$. 
Soit $I$ l'ensemble fini des indices $i$ tels que 
$x\in X_i$. La réunion des $X_i$ pour $i\in I$ est un voisinage ouvert de $x$. 
(En particulier, $I$ est non vide). Pour tout $i\in I$, notons
$V_i$ l'ouvert $\mathfrak X_{i,s}$ de $V$. Par construction, 
$(V_i)$ est un recouvrement ouvert de $V$. 

Soit $\mathfrak U$ un ouvert affine 
de $\mathfrak X$ contenant $\pi(x)$ 
et soit $\sigma$ une section de $\mathfrak E|_{\mathfrak U_s}$. 
Relevons $\sigma$ en une section $s$ de $\mathfrak E$ sur $\mathfrak U$, 
que nous voyons également comme une section de $E$ sur $\mathfrak U_\eta$. 
Si $t$ est une section de $\mathfrak E$ dont la restriction 
$\tau$ à $\mathfrak U_s$ est telle que
$\tau-\sigma$ s'annule en $\pi(x)$ alors $\norm{(t-s)(x)}<1$
si bien que $\widetilde{t(x)}^1=\widetilde{s(x)}^1$ dans
$\widetilde{E(x)}^1$. L'élément $\widetilde{s(x)}^1$ ne dépend donc que
de $\sigma(\pi(x))$, et sa construction est compatible à la restriction 
à un ouvert $\mathfrak V$ de $\mathfrak U$ contenant encore $\pi(x)$. 
On a ainsi construit un morphisme naturel
de $F:=\mathfrak E\otimes \kappa(\pi(x))$
vers $\widetilde{E(x)}^1$, qui induit un morphisme
de $F\otimes_{\kappa(\pi(x))}\hrt x^1$ vers $\widetilde{E(x)}^1$. 
Choisissons $i\in I$. Par construction, les $e_{ij}(\pi(x))$ pour $j$
variable 
forment une base de $F$ ; et comme les
$(e_{ij})$ forment une base orthonormée de $E(x)$, les $\widetilde{e_{ij}(x)}$
forment une base de $\widetilde{E(x)}^1$. Par conséquent, le morphisme
de $F\otimes_{\kappa(\pi(x))}\hrt x^1$ vers $\widetilde{E(x)}^1$ construit
ci-dessus est un isomorphisme ; nous l'utiliserons pour identifier
$F$ à un sous $\kappa(\pi(x))$-espace
vectoriel de $\widetilde{E(x)}^1$. 

On a $\widetilde{(X,x)}^1=\bigcup_{i\in I}\widetilde{(X_i,x)}^1$, et
il résulte du lemme~\ref{redtemk} que 
$\widetilde {(X_i,x)}^1$ est égal à
$\zr{\hrt x^1}{\widetilde k^1}\{V_i\}$ pour
tout $i\in I$ ; c'est donc l'image réciproque de l'ouvert quasi-compact
$\zr {\kappa(\pi(x))}{\widetilde k^1}\{V_i\}$ de 
$\zr {\kappa(\pi(x))}{\widetilde k^1}$. 

Pour tout $i\in I$, la famille $(e_{ij})$ étant une base orthonormée de
$E|_{X_i}$, la famille des $\widetilde{e_{ij}(x)}$ est une base de
$\widetilde{E_x}^1|\widetilde{(X_i,x)}^1$. Comme
$\widetilde{e_{ij}(x)}=e_{ij}(\pi(x))$ pour tout $(i,j)$, 
le lemme \ref{lemm-critere-def-fibre} assure que 
$\widetilde{E_x}^1$ est l'image réciproque
sur $\zr {\hrt x_1}{\widetilde k_1}$
d'un fibré vectoriel $G$
sur 
$\zr {\kappa(\pi(x))}{\widetilde k^1}\{V\}$, caractérisé par le fait 
que sa fibre générique est $F$ et que pour tout $i$, la famille des
$e_{ij}(\pi(x))$ forme une base de $G$ au-dessus de $\zr{\kappa(\pi(x))}
{\widetilde k^1}\{V_i\}$. Mais comme l'image réciproque de 
$\mathfrak E|_V$ sur  $\zr{\kappa(\pi(x))}
{\widetilde k^1}\{V\}$ satisfait ces deux conditions, elle coïncide avec
$G$. Ainsi,  $(V,\mathfrak E)$ est bien un pré-modèle de
$(\widetilde{(X,x)}^1, \widetilde {E_x}^1$). 

\end{proof}

\begin{exem}[Fibré sur une variété propre en valuation triviale]\label{fibvaltriv}
\ \par
Supposons que la valuation de~$k$ est triviale.
Soit $X$ une $k$-variété algébrique et soit $E$ un fibré vectoriel sur $X$ ; 
le paragraphe précédent s'applique et munit le fibré vectoriel~$E_\eta$
sur~$X_\eta$ d'une métrique formelle.
Supposons de plus que $X$ soit propre; 
alors on a $X_\eta=X^\an$ et $E_\eta=E^\an$.

Soit $x$ un point de $X^{\text{an}}\simeq \mathfrak 
X_\eta$ et soit $\tilde x$ son image sur $X\simeq \mathfrak X_{\mathrm s}$. 
Le couple
$(\overline{\{\tilde x\}},E|_{\overline{\{\tilde x\}}})$ est un 
est alors un prémodèle de $(\widetilde{(X^\an,x)}^1, \widetilde {E^{\an}_x}^1)$ (où $\overline{\{\tilde 
x\}}$ est muni de sa structure réduite). 

Notons un cas particulier 
important. Supposons $X$ intègre, soit $\xi$ son point générique, 
et soit~$x$ le point de $X^\an$ correspondant à la 
flèche $\Spec (k(\xi), \abs\cdot_0)\to X$. Le poitn $x$ est alors
l'unique point de $X^\an$ se spécialisant en $\xi$ ; par ce qui précède, $(X,E)$ est alors un prémodèle,
et même un modèle, de $(\widetilde{(X^\an,x)}^1, \widetilde {E^{\an}_x}^1)$. 
\end{exem}

\section{Métriques psh}

\begin{defi}\label{defi.metrique-psh}
Soit $X$ un espace $k$-analytique topologiquement séparé, 
considéré comme espace G-tropical affiné, et soit $L$ un fibré
en droites sur~$X$, localement trivial pour la 
topologie usuelle de~$X$.

On dit qu'une métrique continue sur~$L$ est psh
(\resp  dpsh, \resp tropicalement convexe, \resp tropicalement dconvexe)
\index{métrique!psh}\index{métrique!dpsh}\index{métrique!tropicalement convexe}
si, pour toute section inversible locale~$s$ de~$L$,
la fonction continue $\log\norm s^{-1}$ est psh (définition~\ref{defi.psh})
(\resp dpsh, \resp tropicalement convexe, définition~\ref{defi.trop-convexe},
\resp tropicalement d-convexe).
\end{defi}

Pour qu'il en soit ainsi, il suffit que $L$ possède
une telle section au voisinage de tout point de~$X$.

Il découle de ces définitions qu'une métrique tropicalement convexe 
est tropicalement d-convexe, qu'une métrique psh est d-psh,
et qu'une métrique tropicalement convexe, \resp tropicalement d-convexe,
est psh, \resp d-psh.

Si la métrique de~$L$ est psh, sa courbure $c_1(\overline L)$
est un courant positif.

Comme toute fonction lisse sur~$\R^n$
est localement différence de deux fonctions psh lisses
(proposition~\ref{prop.diff-positive}, \emph b)),
toute métrique lisse sur~$L$ est tropicalement d-convexe.

Si $X$ est presque bon, toute fonction G-lisse est localement
différence de deux fonctions tropicalement d-convexes
(proposition~\ref{rema.G-lisse-diff-psh}, \emph b)),
de sorte que toute métrique G-lisse
est tropicalement d-convexe. 
C'est en particulier le cas des métriques paralinéaires.

\begin{prop}\label{prop.metrique-pl-approchable} 
Soit $X$ un espace $k$-analytique séparé et presque bon,
et soit~$L$ un fibré en droites sur~$X$ 
muni d'une métrique paralinéaire entière. 
On suppose que $L$ est localement trivial pour la topologie usuelle de~$X$.

Soit $x\in X$ tel qu'il
existe une puissance
strictement positive du fibré résiduel~$\widetilde{L_x}$
engendrée par ses sections globales.
Alors la métrique de~$L$ est tropicalement convexe (et donc psh)
au voisinage de~$x$.
\end{prop}
\begin{proof}
Quitte à remplacer~$L$ par une puissance strictement positive,
on suppose que $\widetilde{L_x}$ est engendré par ses sections globales.
D'après le
corollaire \ref{coro-engendre-sections-globales},
il existe donc une famille finie $(a_1,\dots,a_m)$ 
d'éléments de~$\mathscr O_{X,x}^\times$
telle que 
la fonction $\norm{s}^{-1}$ soit égale à $\max(\abs{a_1},\dots,\abs{a_m})$
au voisinage de~$x$. 
Il découle alors de la proposition~\ref{exem.psh-approchable}
que la métrique de~$L$ est tropicalement convexe au voisinage de~$x$.
\end{proof}

\begin{rema}
Voyons comment la proposition précédente
permet de redémontrer le fait rappelé après la définition~\ref{defi.metrique-psh}
que les métriques paralinéaires strictement entières
sont tropicalement d-convexes (et donc d-psh).

L'assertion est locale. 
Soit $x\in X$. La proposition~\ref{lemme-difference-globalement-engendre}
assure qu'il existe deux fibrés en droites~$M_1$ et~$M_2$
sur $\widetilde{(X,x)}$, engendrés par leurs sections globales, 
tels que $\widetilde{L_x}\simeq M_1\otimes M_2^\vee$. 

Relevons $M_1$ et $M_2$ en deux fibrés 
métrisés paralinéaires~$L_1$ et~$L_2$ sur~$(X,x)$,
strictement entiers, et 
triviaux comme fibrés en droites au voisinage de~$x$
(théorème~\ref{theo-fibre-residuel});
on dispose aussi d'après ce théorème
d'un isomorphisme $L\simeq L_1\otimes L_2^\vee$ 
de fibrés en droites métrisés au voisinage de~$x$.
D'après le premier cas, les métriques de~$L_1$ et~$L_2$ sont tropicalement convexes
au voisinage de~$x$. Par conséquent, la métrique de~$L$ est tropicalement d-convexe
au voisinage de~$x$.
\end{rema}

\begin{prop}\label{prop.nef-vers-psh}
Soit $X$ un $k$-espace analytique séparé et sans bord, 
et soit~$L$ un fibré en droites sur~$X$ muni d'une métrique
paralinéaire strictement entière. 

Soit $x\in X$ tel que le fibré résiduel~$\widetilde {L_x}$
soit numériquement effectif (\ref{def-numeff}). 
Alors la métrique de~$L$
est psh au voisinage de~$x$.
\end{prop}
\begin{proof}
Soit $(V,M)$ un prémodèle du fibré résiduel~$\widetilde L(x)$;
comme $X$ est sans bord, $V$ est propre;
par le lemme de Chow, on peut supposer que $V$ est projective.
Soit $M'$ un fibré ample sur~$V$; il 
existe un voisinage compact~$U$ de~$x$ dans~$X$ 
et une métrique strictement~\pl sur le fibré trivial~$\mathscr O_U$
dont le fibré résiduel en~$x$ est~$M'$; notons $h=\log(\norm{1}^{-1})$
pour~$M'$; d'après la fonction précédente, 
c'est une fonction psh (et même tropicalement convexe) sur~$X$.

Pour tout nombre réel~$\eps>0$, on peut munir
le fibré en droites~$L|_U$ de la métrique $\norm\cdot_\eps$
qui est le produit de sa métrique initiale par la fonction $\exp(-\eps h)$.
Lorsque $\eps$ tend vers~$0$, cette métrique tend uniformément vers celle de~$L$
sur~$U$.
Il suffit donc d'observer que cette métrique est psh lorsque~$\eps$
est un nombre rationnel (a posteriori, elle sera psh pour tout nombre
réel positif~$\eps$).

Supposons donc~$\eps$ rationnel; écrivons $\eps=m/n$ avec $m,n\geq 1$. 
Alors la métrique~$\norm\cdot_\eps$ est strictement~$\Q$-\pl;
la métrique $(\norm\cdot_\eps)^N$ sur~$L^{\otimes N}|_U$ est strictement~\pl;
il suffit de prouver qu'elle est~psh.
Or, le fibré résiduel associé est donné par
le prémodèle $(V,M^{\otimes n}\otimes (M')^{\otimes m})$.
Comme $M$ est numériquement effectif et $M'$ est ample,
il résulte du critère de Nakai-Moishezon 
que ce fibré en droites sur~$V$ est ample.
D'après la proposition précédente, la métrique considérée est~psh.
\end{proof}

\subsection{}
On dit qu'une métrique continue~$\norm\cdot$ sur~$L$ 
est globalement psh\index{métrique!globalement psh}
s'il existe une suite $(\norm\cdot_n)$ de métriques psh lisses  sur~$L$
telles que les fonctions continues $\log (\norm\cdot/\norm\cdot_n)$
sur~$X$ convergent uniformément vers~$0$.

On dit qu'une métrique continue sur~$L$ est 
globalement dpsh\index{métrique!globalement dpsh}
s'il existe deux fibrés en droites~$L_1$ et~$L_2$
sur~$X$, munis de métriques globalement psh,
tels que $\bar L\simeq \bar L_1\otimes\bar L_2^{-1}$.

\begin{prop}
Sur l'espace projectif~$\P^n$, le fibré en droites~$\mathscr O(1)$
muni de sa métrique formelle canonique est globalement psh
et tropicalement convexe.

Plus précisément, il existe une famille $(\norm\cdot_\eta)_{\eta>0}$
de métriques psh lisses sur~$\mathscr O(1)$, qui converge
uniformément vers sa métrique formelle canonique quand $\eta\to 0$,
vérifiant de plus la propriété suivante:
soit $x\in\P^n$, soit $I$ l'ensemble des $i\in\{0,\dots,n\}$ tels
que $x\in D(T_i)$; il existe
un voisinage~$U$ de~$x$ dans $\bigcap_{i\in I} D(T_i)$
tel que, pour $i\in I$ et $\eta<1$, la fonction $\norm{T_i}_\eta$
est tropicalisée par le moment $(T_j/T_i)_{j\in I}$
sur~$U$.
\end{prop}
\begin{proof}
Tout fibré résiduel de ce fibré en droites admet
un prémodèle qui est le fibré~$\mathscr O(1)$ d'une sous-variété
de~$\P^n_{\tilde k}$; il est en particulier semi-ample.
D'après la proposition~\ref{prop.metrique-pl-approchable}, 
la métrique de Weil est tropicalement convexe. 
Démontrons qu'elle est globalement psh.

Pour tout entier~$p\geq 0$ et tout nombre réel~$\eta>0$,
Demailly définit dans~\cite{demailly2009}, lemme~5.18,
une fonction~$M_\eta\colon\R^p\ra\R$ vérifiant les propriétés
suivantes:
\begin{enumerate}\def\theenumi{\alph{enumi}}\def\labelenumi{\theenumi)}
\item La fonction~$M_\eta$ est lisse, convexe et croissante en chaque variable;
\item Pour tout $(t_1,\dots,t_p)\in\R^p$,
\[ \max(t_1,\dots,t_p)\leq M_\eta(t_1,\dots,t_p)\leq \max(t_1,\dots,t_p)+\eta;\]
\item Si $t_j+\eta\leq \max_{i\neq j}(t_i-\eta)$, on a
\[ M_\eta(t_1,\dots,t_p)=M_\eta(t_1,\dots,\widehat{t_j},\dots,t_p)\ ;\]
\item Pour tout $t\in\R$,
\[ M_\eta(t_1+t,\dots,t_p+t)=M_\eta(t_1,\dots,t_p)+t .\]
\end{enumerate}

Il résulte de la lissité de la fonction~$M_\eta$
et de la propriété~d) que l'on définit 
une métrique lisse~$\norm\cdot_\eta$ 
sur~$\mathscr O(1)$ en posant 
\[ \log\norm{s_P}_\eta^{-1}([x_0:\dots:x_n])
 = - \log\abs{P(x_0,\dots,x_n)}
 + M_\eta(\log\abs{x_0},\dots,\log\abs{x_n}) \]
pour tout  polynôme~$P$ homogène de degré~$1$.
Comme
\[  \log\norm{s_P}_\eta^{-1}-\log\norm{s_P}^{-1}
 = 
  M_\eta(\log\abs{x_0},\dots,\log\abs{x_n}) 
- \max(\log\abs{x_0},\dots,\log\abs{x_n} ),
\]
la propriété~b) ci-dessus entraîne que  
$  \log\norm{s_P}_\eta^{-1}-\log\norm{s_P}^{-1}$
appartient à l'intervalle~$[0,\eta]$.
Par suite, la métrique $\norm{\cdot}_\eta$ converge uniformément
vers $\norm{\cdot}$ lorsque $\eta\ra 0$.


Soit $x=[x_0:\dots:x_n]$ un point~ de~$\P^n$; supposons pour fixer
les idées que $x_0,\dots,x_p\neq 0$ mais $x_{p+1}=\dots=x_n=0$,
et prenons pour $P$ le polynôme~$X_0$.
Pour $j\in\{1,\dots,n\}$, soit $f_j=T_j/T_0$; c'est
une fonction holomorphe sur $D(T_j)$.
Les fonctions~$f_1,\dots,f_p$ sont inversibles sur
tout voisinage~$U$ de~$x$ contenu dans~$D(T_0\dots T_p)$.
Dans un tel  voisinage, on a
\begin{align*}
 \log\norm{s_P}_\eta^{-1}(x)
& = -\log\abs{x_0}+
  M_\eta(\log\abs{x_0},\dots,\log\abs{x_n}) \\
& = M_\eta (0,\log\abs{f_1(x)},\dots,\log\abs{f_p(x)}) \end{align*}
si $\eta<1$ et si $U$ est pris assez petit pour que l'on ait
\[ \sup_U\sup_{j> p} \log \abs{f_j}+1 \leq \inf_U \max_{i\leq p}\log\abs{f_i}-1. \]
Comme $M_\eta$ est lisse et convexe sur~$\R^{p+1}$,
cette formule prouve que la métrique $\norm\cdot_\eta$
est lisse et psh sur~$U$, tropicalisée par le moment~$(f_1,\dots,f_p)$.

La proposition en résulte.
\end{proof}

\begin{coro}\label{coro.formel-glob-psh-approchable}
Soit $X$ un espace strictement $k$-analytique propre, 
soit $\mathfrak X$ un $k^\circ$-modèle propre de~$X$,  
soit $\mathfrak L$ un fibré en droites sur~$\mathfrak X$
et soit $\bar L$ 
sa fibre générique, muni de sa métrique formelle canonique.
Si $\mathfrak L$ est semi-ample, par
alors $\bar L$ est globalement psh.
\end{coro}
\begin{proof}
Par hypothèse, il existe un entier $d\geq 1$ et un morphisme
$p\colon\mathfrak X\ra\P^N_{k^\circ}$
tel que $p^*\mathscr O(1)$ soit isomorphe à~$\mathfrak L^{\otimes d}$.
Cela prouve que la métrique sur~$L^{\otimes d}$ déduite
de la métrique formelle de~$L$
est l'image réciproque par~$p$ de la métrique
de Weil sur~$\mathscr O(1)$.
Le lemme en résulte.
\end{proof}

\begin{coro}
Soit $L$ un fibré en droites sur un espace $k$-analytique
topologiquement séparé~$X$.
Si $L$ est semiample, alors $L$ possède une métrique lisse psh.
\end{coro}
\begin{proof}
Par hypothèse, il existe un entier $d\geq 1$
et un morphisme  $p\colon X\ra\P^n_k$
tel que $p^*\mathscr O(1)$ soit isomorphe à $L^{\otimes d}$.
L'image réciproque d'une métrique lisse psh sur~$\mathscr O(1)$
induit une métrique lisse psh sur~$L^{\otimes d}$,
et donc aussi sur~$L$.
\end{proof}

\begin{coro}\label{coro.formel-glob-approchable}
Soit $X$ l'analytifié d'une $k$-variété projective
et  soit $\bar L$ un fibré en droites sur~$X$ muni d'une métrique paralinéaire.
Alors $\bar L$ est globalement d-psh.
\end{coro}
\begin{proof}
Soit $(\mathfrak X,\mathfrak L)$ un modèle formel de~$(X,L)$
induisant la métrique formelle donnée de~$\bar L$.
Comme $X$ possède un modèle formel qui est projectif,
on peut supposer que $\mathfrak X$ est projectif sur~$k^\circ$.
Il résulte alors du théorème
de comparaison géométrie algébrique/géométrie formelle
(théorème~2.13.8 de~\cite{abbes2010})
que $\mathfrak L$ est un fibré en droites algébrique sur~$\mathfrak X$.
Dans ces conditions, il existe deux fibrés en droites amples $\mathfrak L_1$
et $\mathfrak L_2$ sur~$\mathfrak X$
tels que $\mathfrak L\simeq \mathfrak L_1\otimes \mathfrak L_2^{-1}$.
Pour $i\in\{1,2\}$, soit $\bar L_i$ la fibre générique de~$\mathfrak L_i$,
muni de sa métrique formelle canonique ; c'est une métrique
globalement psh-approchable 
et tropicalement psh-approchable
d'après le corollaire~\ref{coro.formel-glob-psh-approchable}.
L'isomorphisme $\mathfrak L\simeq\mathfrak L_1\otimes\mathfrak L_2^{-1}$
induit un isomorphisme $ L\simeq  L_1\otimes  L_2^{-1}$.
C'est même une isométrie, d'où le corollaire.
\end{proof}

\begin{rema}
Il serait intéressant de prouver que toute métrique psh continue sur une variété
projective est globalement psh.
Cela est vrai  en analyse complexe, cf.~\cite{maillot2000}, théorème~4.6.1,
ou aussi \cite{blocki-kolodziej2007}, théorème~1.

Par ailleurs, on pourrait vouloir  généraliser 
le corollaire~\ref{coro.formel-glob-approchable} à 
un espace strictement $k$-analytique propre possédant
un modèle formel à fibre spéciale projective. 
Dans~\cite{Scholze-2018}, Scholze propose en effet de considérer
les espaces analytiques non archimédiens satisfaisant cette condition
comme les analogues des variétés complexes kählériennes.
\end{rema}

\chapter{Produits de courants de courbure}

\def\cycle{\operatorname{cycle}}

\section{Courbure d'un fibré en droites métrisé}

\subsection{}
Soit $X$ un $n$-espace $k$-analytique sans bord 
et soit~$L$ un fibré en droites sur~$X$ muni d'une métrique continue.

Soit~$U$ un ouvert de $X$ sur lequel $L$ admet 
une section inversible~$s$.
La fonction  $\log\norm s^{-1}$ est alors continue sur~$U$.
Si $t$ est une autre section inversible de~$L$ sur~$U$,
le quotient $t/s$ est une fonction inversible~$f$ sur~$U$
et l'on a $\log \norm{s}^{-1} = \log \norm{t}^{-1} + \log \abs f$.

Ainsi, ces fonctions continues $(\log \norm s^{-1})$
se recollent en une section globale 
du faisceau quotient $\mathscr C_X/\log(\abs{\mathscr O_X^\times})$.

Lorsque la métrique est lisse, on obtient une section globale
du sous-faisceau $\mathscr A_X^0/\log(\abs{\mathscr O_X^\times})$.

Lorsqu'elle est~\pl, on obtient une section globale
du sous-faisceau $\PL_X/\log(\abs{\mathscr O_X^\times})$.

\subsection{}
Soit $X$ un $n$-espace $k$-analytique sans bord 
et soit~$\bar L$ un fibré en droites
sur~$X$ muni d'une métrique continue.

L'application $u\mapsto \ddc [u]$ de $\mathscr C_X$ dans l'espace
des courants de bidegré~$(1,1)$ sur~$X$ passe au quotient
par le sous-espace $\log(\abs{\mathscr O_X^\times})$
qui sont lisses et vérifient $\ddc u=0$.
On déduit donc de la section globale du faisceau quotient
$\mathscr C_X/\log(\abs{\mathscr O_X^\times})$ associée à~$\bar L$
un courant de bidegré~$(1,1)$ sur~$X$ 
que l'on note $c_1(\bar L)$.

Sur un ouvert~$U$ de $X$ sur lequel $L$ admet 
une section inversible~$s$, ce courant est donné 
par $\ddc \log \norm{s}^{-1}$.

\begin{defi}\label{defi.courant-courbure}
On dit que ce courant $c_1(\bar L)$ est le courant de courbure
du fibré inversible métrisé~$\bar L$.
\end{defi}

Sa formation est compatible à la restriction.

Le courant de courbure d'un fibré inversible métrisé~$\bar L$
est $\di$-fermé et $\dc$-fermé.

Si la métrique de~$L$ est lisse, son courant de courbure
est une \emph{forme} lisse.

Si la métrique de~$L$ est psh, son courant de courbure
est positif.

\subsection{}
En géométrie complexe, il est souvent important d'utiliser
des métriques singulières. 
Les constructions précédentes s'étendent naturellement dans ce cadre,
pour peu qu'on considère des semi-métriques localement intégrables.

Soit $X$ un $n$-espace $k$-analytique sans bord 
et soit~$L$ un fibré en droites sur~$X$ muni d'une semi-métrique.
On dit que cette semi-métrique est localement intégrable
si pour toute section inversible~$s$ de~$L$ sur un ouvert~$U$
de~$X$, la fonction $\log \norm s$ sur~$U$
(à valeurs dans $\R\cup\{-\infty\}$) est localement intégrable.

Par analogie avec la définition~\ref{defi.metrique-psh},
on dit que cette semi-métrique est psh si, pour toute section
inversible~$s$ de~$L$ sur un ouvert~$U$ de~$X$, la fonction~$\log \norm s$
sur~$U$ est psh.

Dans les deux cas, il suffit qu'il existe, au voisinage de tout point de~$X$,
une section inversible de ce type, car les fonctions du type~$\log \abs f$,
pour une fonction  analytique inversible~$f$, sont continues et psh.

On déduit de même une section globale du faisceau quotient
$\mathrm L^1_{\loc,X}/\log(\abs{\mathscr O_X^\times})$,
respectivement du faisceau quotient
$\PSH_X/\log(\abs{\mathscr O_X^\times})$,
ainsi que le courant de courbure $c_1(\overline L)$
du fibré en droites semi-métrisé~$\overline L$, à semi-métrique
localement intégrable. Lorsque cette semi-métrique est psh,
ce courant est positif.

\subsection{}
Soit $X$ un $n$-espace $k$-analytique sans bord 
et soit~$L$ un fibré en droites sur~$X$.
Comme en géométrie algébrique (\cite{ega4.4}, 20.1.8), 
le $\mathscr M_X$-module $L \otimes_{\mathscr O_X} \mathscr M_X$
est inversible; ses sections sont appelées sections méromorphes;
ses sections inversibles sont appelées sections méromorphes régulières.

Si $s$ est une section méromorphe (régulière) de~$L$,
alors pour tout ouvert~$U$ et toute section inversible~$\eps$
de~$L$ sur~$U$, le quotient $s|_U/\eps$ est une fonction méromorphe
(inversible) sur~$U$.

Une section méromorphe régulière~$s$ possède un diviseur~$\div(s)$,
construit localement dans des cartes trivialisantes.
Son support est un fermé de Zariski purement de codimension~$1$,
sur le complémentaire~$U$ duquel $s$ fournit une trivialisation de~$L$.

Supposons maintenant que $L$ soit muni d'une semi-métrique localement
intégrable.
Alors la fonction $\log \norm{s}^{-1}$ est définie et localement
intégrable sur~$U$;
d'après le paragraphe~\ref{ss.prolongement-courant},
elle définit alors un courant sur~$X$  que l'on note $[\log \norm{s}^{-1}]$.

\begin{theo}[Formule de Poincaré--Lelong]\label{theo.poincare-lelong-metriques}
Soit $X$ un $n$-espace $k$-analytique sans bord
et soit~$\bar L$ un fibré en droites sur~$X$ muni d'une semi-métrique 
localement intégrable.
Soit $s$ une section méromorphe régulière de~$L$ sur~$X$.
On a l'égalité
\[ \ddc [\log\norm s^{-1}]+\delta_{\div(s)}=c_1(\bar L). \]
\end{theo}
\begin{proof}
Cette égalité de courants se vérifie localement,
sur un ouvert~$U$ sur lequel $L$ possède une section inversible~$\eps$.
Soit $f$ l'unique fonction méromorphe inversible sur~$U$ 
telle que $s|_U = f\eps$. 
La formule indiquée résulte alors des définitions et de 
la formule de Poincaré--Lelong (théorème~\ref{pcrllg})
pour les fonctions méromorphes régulières.
\end{proof}

\begin{coro}
Soit $X$ une $k$-courbe analytique propre et
soit $L$ un fibré en droites sur~$X$ muni d'une semi-métrique
localement intégrable.
La masse totale du courant~$c_1(\overline L)$
est égale au degré du fibré en droite~$L$.
\end{coro}
Rappelons que $X$ est algébrisable, de même que~$L$,
plus précisément qu'il existe, à isomorphisme près, une unique
$k$-courbe algébrique muni d'un fibré en droite
redonnant $(X,L)$ après analytification.
(Pour une démonstration de l'algébrisation d'une courbe, 
voir~\cite{ducros-2024b}, théorème~3.7.2. L'algébrisation
de~$L$ et la propriété d'unicité se déduisent alors
du théorème GAGA dans ce contexte ; voir l'appendice~A
de~\cite{poineau2010a} pour une démonstration détaillée.)

\begin{proof}
Commençons par traiter le cas où $X$ est intègre.
Le fibré en droites~$L$ possède alors une section méromorphe
globale non nulle~$s$, et cette section est régulière car $X$ est intègre.
D'après la formule de Poincaré--Lelong, on a alors
\[ \langle c_1(\overline L), 1\rangle
 = \langle \ddc [\log \norm{s}^{-1}], 1\rangle 
 + \langle \delta_{\div(s)}  , 1\rangle 
=  \langle [\log \norm{s}^{-1}], \ddc (1)\rangle 
 + \langle \delta_{\div(s)}  , 1\rangle 
. \]
Le premier terme est nul, car $\ddc(1)=0$.
Si $\div(s)=\sum n_i [P_i]$, alors $\deg(L)=\sum n_i [k(P_i):k]$
et $\langle \delta_{P_i},1\rangle = [k(P_i):k]$.
Cela prouve l'égalité $\langle c_1(\overline L), 1\rangle=\deg(L)$
dans le cas où $X$ est intègre.

Traitons maintenant le cas général.
Soit $(X_i)$ la famille finie des composantes irréductibles (intègres) de~$X$
et, pour tout~$i$, soit $m_i$ la multiplicité générique de~$X_i$ dans~$X$.
Soit $Y$ la courbe somme disjointe des~$X_i$ et soit $p\colon Y\to X$
le morphisme fini canonique;
on considère les~$X_i$ comme des sous-espaces ouverts et fermés de~$Y$.

Le fibré en droites semi-métrisé~$\overline{L}$ définit
par image inverse un fibré en droites semi-métrisé~$\overline M$
sur~$Y$.
Soit $(U_j)$ un recouvrement ouvert fini de~$X$ tel que 
pour tout~$j$, le fibré en droites $L|_{U_j}$ 
possède une section inversible~$s_j$.
Soit $(\lambda_j)$ une partition de l'unité lisse subordonnée
à ce recouvrement.
Alors 
\[
\langle c_1(\overline L), 1\rangle 
= \sum_j \langle c_1(\overline L),\lambda_j \rangle
= \sum_j \int_{U_j} \log \norm{s_j}^{-1} \ddc \lambda_j . \]
Pour tout~$i$, soit $V_i$ l'ouvert de~$X$ 
complémentaire de la réunion des autres composantes composantes que~$X_i$;
ces ouverts sont disjoints et le complémentaire de leur réunion
est un ensemble fini de points rigides.
De plus, $p^{-1}(V_i)$ est un ouvert de~$X_i$, dont le complémentaire
est un ensemble fini de points rigides; il s'identifie à~$V_{i,\red}$.
Comme la $(1,1)$-forme lisse~$\ddc \lambda_j$ 
est nulle au voisinage des points rigides,
on a 
\[ \int_{U_j} \log \norm{s_j}^{-1} \ddc \lambda_j
 = \sum_i \int_{U_j \cap V_i} \log \norm{s_j}^{-1} \ddc \lambda_j. \]
Pour tout~$j$ et tout~$i$, on a
\[ \int_{U_j \cap V_i} \log \norm{s_j}^{-1} \ddc \lambda_j 
 = m_i \int_{p^{-1}(U_j) \cap X_i} \log \norm{p^*s_j} \ddc (p^*\lambda_j).
\]
Alors, 
\[
\langle c_1(\overline L), 1\rangle 
 = \sum_{j,i} m_i \int_{p^{-1}(U_j) \cap X_i}
        \log \norm{p^* s_j}^{-1} \ddc (p^*\lambda_j) \\
 = \sum_{i} m_i \langle c_1(\overline L|{X_i}), 1\rangle, 
\]
en renversant les calculs.
Par le cas intègre, on a $\langle c_1(\overline L|_{X_i}),1\rangle
=\deg(L|_{X_i})$.
Par suite,
\[ \langle c_1(\overline L), 1\rangle 
= \sum_i m_i \deg(L|_{X_i}) = \deg(L). \] 
Cela conclut la démonstration.
\end{proof}

\section{Le courant canonique d'une courbe}

Dans ce paragraphe, on appelle \emph{courbe $k$-analytique}
un espace $k$-analytique séparé, purement de dimension~1.
Une telle courbe peut donc être considérée comme espace tropical de niveau~$1$.

\subsection{}
Dans~\cite{Temkin-2016}, M.~Temkin a défini,
pour tout espace espace $k$-analytique~$S$ et tout
$S$-espace analytique
quasi-lisse~$X$ la « semi-norme de Kähler »
sur le fibré vectoriel~$\Omega^1_{X/S}$.
Sa construction fournit,
pour tout domaine~$V$ de~$X$ et toute section $\omega\in\Omega^1_{X/S}(V)$,
une fonction $x\mapsto \norm{\omega(x)}$ sur~$V$.
Rappelons sa construction. Soit $x\in X$
et soit $s$ son image sur $S$. 
L'application du module
des différentielles de Kähler algébriques $\Omega^1_{\hr x/\hr s}$ 
vers $\R_+$ qui envoie $\eta$ sur la borne inférieure des $\sup \abs {a_i} \cdot \abs{b_i}$
pour toutes les expressions $\eta = \sum a_i \, db_i$,
où $a_i, b_i \in \hr x$, est une semi-norme ; notons $\widehat{\Omega^1}_{\hr x/\hr x}$
le séparé-complété correspondant. 
Pour tout domaine analytique $V$ de $X$ contenant $x$ et tout $\omega\in \Omega^1_{V/S}$, 
on obtient par tiré en arrière le long du morphisme $x\colon \mathscr M(\hr x)\to V$ 
un élément $\omega(x)$ du $\hr x$-espace de Banach $\widehat{\Omega^1}_{\hr x/k
\hr s}$, et c'est sa norme que l'on note $\norm{\omega(x)}$.

C'est une notion différente de la notion de semi-métrique
au sens du présent travail, car elle n'est pas invariante par extension
des scalaires. Toutefois, il résulte
de l'exemple~\ref{exem.metrique-rang-1} qu'elle fournit,
au moins dans le cas des courbes,
une semi-métrique canonique (en notre sens) sur le fibré canonique.
Nous allons voir que cette semi-métrique est localement intégrable
et, lorsque la courbe est lisse, calculer son courant de courbure.

\begin{exem}
Prenons l'exemple de la droite affine $X=\mathbf A^1_k=\mathscr M(k[T])$
sur un corps algébriquement clos~$k$.
La forme différentielle $dT$ est une base de $\Omega^1_{X/k}$.
D'après l'exemple~6.2.1 de~\cite{Temkin-2016}, pour tout
point~$x$ de~$X$, sa norme en~$x$ est le \emph{rayon}~$r(x)$ de~$x$, 
borne inférieure des rayons des disques de~$X$ qui contiennent~$x$,
ou encore la distance, dans~$\hr x$, de $T(x)$ à~$k$.

On voit déjà que $\norm{\da T}$ s'annule exactement sur l'ensemble~$X(k)$
des points rigides; comme cet ensemble n'est pas fermé, 
car tout voisinage ouvert du point de Gauss~$\eta_1$
contient des points rigides,
la semi-métrique de Temkin n'est pas continue.
(Néanmoins, cette fonction $\norm{\da T}$ est semi-continue supérieurement.)

Soit $a\in k$ et soit $r$ un nombre réel strictement positif ;
on a \[\norm{\da T(\eta_{a,r})}=r=\abs{T-a}(\eta_{a,r})\]
puisque $\abs{T(\eta_{a,r})-b}=\sup (\abs{b-a},r)$
pour tout $b\in k$.
Ainsi, la fonction $r\mapsto \log(\norm{\da T(\eta_{a,r}))}$
est affine sur la partie squelettique de~$X$
formée des points du type~$\eta_{a,r}$, pour $r>0$.
En particulier, en restriction à toute partie squelettique de~$X$,
la fonction $\log(\norm{\da T})$ est paralinéaire et donc continue.
Cela démontre dans ce cas
que la semi-métrique de Temkin est localement intégrable.
\end{exem}

\subsection{}
Soit $X$ une courbe $k$-analytique et
soit $W$ un domaine analytique de~$X$. 
On dit qu'une fonction $t\in\mathscr O_X(W)$ est un paramètre 
si le morphisme de~$W$ vers~$\A^1$ induit par~$t$ est quasi-étale.
Cela entraîne que pour tout point $x\in W$
l'extension $\hr {t(x)}\to \hr x$ est finie séparable.
On dit que $t$ est un paramètre modéré (resp. non ramifié)
si, de plus, ces extensions sont toutes modérément ramifiées 
(resp. non ramifiées).
Par exemple,
la fonction coordonnée de la droite affine en est un paramètre non ramifié.

On dit que le paramètre~$t$ est \emph{monomial} en un point~$x$ de~$W$
si $t(x)$ appartient au squelette canonique de~$\gm$.

\begin{prop}
Soit $k$ un corps ultramétrique complet, non trivialement valué
et algébriquement clos~$k$. 
Soit $X$ une courbe $k$-analytique quasi-lisse,
soit $V$ un domaine connexe de~$X$, 
et soit $\omega\in\Omega^1_{X/k}(V)$ 
une forme différentielle non nulle.

\begin{enumerate}
\item Pour tout point $x\in V$, on a $\norm{\omega(x)}$ si et seulement
si $x$ est un point rigide de~$V$.
\item Pour toute partie squelettique~$\Sigma$ de~$V$,
la restriction à~$\Sigma$ de la fonction $\log( \norm{\omega})$ 
est paralinéaire.
\item
En particulier, la semi-métrique canonique est localement intégrable.
\end{enumerate}
\end{prop}
\begin{proof}
On peut déduire du théorème de réduction semi-stable
que $V$ possède un G-recouvrement par des domaines sur lesquels il existe
un paramètre non ramifié.  Voir aussi 
le théorème~6.3.1 de~\cite{Temkin-2010} et
le lemme~3.6.14 de~\cite{CohenTemkinTrushin-2016}
pour une approche plus ponctuelle, 
mais qui repose également sur le théorème de réduction semi-stable.


Comme $\omega$ n'est pas nulle et comme $V$ est connexe,
la restriction de~$\omega$ à tout domaine analytique connexe non vide
de~$V$ est encore non nulle (et son lieu des zéros
est un ensemble fermé et discret de points rigides).
Pour démontrer la proposition, on peut donc se restreindre aux
domaines d'un G-recouvrement donné de~$V$,et en conséquence supposer
qu'il existe un paramètre non ramifié~$t$ sur~$V$.

Puisque $t$ est quasi-étale, $dt$ est une base du fibré en droites~$\Omega^1_{V/k}$; il existe donc une unique fonction $f\in\mathscr O_X(V)$
telle que $\omega = f\,\da t$. 
Comme $\omega\neq0$ et que
$V$ est connexe, la fonction~$f$ 
n'est pas nulle sur~$V$.
Si $x$ est un point rigide, $\norm{\omega(x)}=0$ par construction 
de la semi-métrique de Temkin, car $\Omega_{\hr x/k}=0$.
Sinon, on a $f(x)\neq 0$; de plus,
le théorème~2.3.2 de~\cite{CohenTemkinTrushin-2016}
affirme que $\norm{\da t}(x)$ est la distance, notée~$r_{t_x}$
de~$t(x)$ à~$k$ dans~$\hr x$; en particulier, $\norm{\da t(x)}>0$,
si bien que $\norm{\omega(x)}>0$. L'assertion~\emph a) est donc démontrée.

Soit $\Sigma$ une partie squelettique de~$V$.
Puisque $f$ est inversible au voisinage de~$\Sigma$,
la fonction $\log\norm f|_\Sigma$ est paralinéaire.
D'autre part, $\norm {\da t}$ est la composée
de la fonction rayon sur~$\A^1$ et du morphisme~$f$ induit par
le paramètre~$t$.
Comme $f(\Sigma)$ est une partie squelettique de~$\A^1_k$
(lemme~3.6.8 de~\cite{CohenTemkinTrushin-2016}),
l'assertion découle du fait que le logarithme de la fonction rayon induit
une fonction paralinéaire
sur toute partie squelettique. Cela démontre l'assertion~\emph b).

L'assertion~\emph c) découle de~\emph b) puisque les fonctions
paralinéaires sur un squelette sont continues.
\end{proof}

\subsection{}
Soit $k$ un corps ultramétrique complet, non trivialement valué,
algébriquement clos.
Soit $X$ une courbe analytique quasi-lisse sur~$k$.
Soit $\Sigma$ une partie squelettique compacte de~$k$.

On sait que $\Sigma$ possède une \emph{triangulation},
c'est-à-dire une décomposition cellulaire finie~$\mathscr C$
telle que, pour toute arête~$e\in\mathscr C_1$ de~$\mathscr C$,
il existe un ouvert~$U_e$ de~$X$ isomorphe à une couronne $\{r<\abs T<R\}$
qui identifie l'arête ouverte au squelette canonique de cette couronne,
c'est-à-dire à l'ensemble des points~$\eta_t$, 
pour $t\in\mathopen]r;R\mathclose[$.
Pour tout point~$x$ de type~II de~$X$, 
le corps~$\hrt x^1$ est le corps des fonctions
d'une unique courbe projective lisse~$C_x$
sur le corps algébriquement clos~$\tilde k^1$;
on note $g(x)$ le genre de cette courbe.
Si $x$ est un point de type~III, on pose $g(x)=0$.

Cet énoncé peut déduire du théorème de réduction semi-stable,
il peut également être démontré directement (voir~\cite{ducros-2024b})
et servir de base à une démonstration de ce théorème.

Pour tout sommet~$x\in\mathscr C_0$, on note $v(x)$
le nombre d'arêtes de~$\mathscr C$ qui contiennent~$x$.

On définit un « diviseur » sur~$\Sigma$ par la formule 
\[ K_{\mathscr C} = \sum_{x\in\mathscr C_0} (2g(x)-2+v(x))  \delta_x ,\]
où $x$ est la mesure de Dirac en~$x$.
Raffinons $\mathscr C$ en une décomposition $\mathscr C'$
par adjonction d'un point~$x$ à l'intérieur d'une arête. 
Si le point $x$ est de type~III, on a $g(x)=0$
par définition; sinon il appartient au squelette canonique d'une couronne
ouverte,
de sorte que 
le corps~$\hrt x$ est isomorphe à~$\tilde k^1(T)$, si bien que 
$C_x=\P^1_{\tilde k^1}$ et que $g(x)=0$. 
Et puisque $x$ est situé à l'intérieure d'une arête, on a $v(x)=2$. Par conséquent
$2g(x)-2+v(x)=0$, d'où il résulte que le diviseur~$K_{\mathscr C'}$
est égal au diviseur~$K_{\mathscr C}$.

Il s'ensuit que le diviseur~$K_{\mathscr C}$ ne dépend pas du choix
de la triangulation~$\mathscr C$. On notera $K_\Sigma$
ce diviseur, et on dira que c'est le diviseur canonique de~$\Sigma$.

Si $\Sigma'$ est une partie squelettique compacte qui contient~$\Sigma$,
on a $K_\Sigma \leq K_{\Sigma'}|_\Sigma $, mais l'inégalité n'est
pas une égalité en général.

\subsection{}
Soit $k$ un corps valué complet, non trivialement valué, algébriquement clos
et soit $X$ une courbe analytique lisse sur~$k$.
Soit $\omega$ une forme différentielle méromorphe régulière sur~$X$.

\subsection{}
Soit $x\in X$ un point de type~II et soit $e\subset X$ 
une $1$-cellule dont $x$ est une extrémité.

Il existe un point $y\in e$ et un
ouvert~$U$ de~$X$
tel que $U$ soit isomorphe à une couronne ouverte
et ait pour squelette l'intervalle~$\mathopen]x;y\mathclose[$ de~$e$.
Soit $t$ une fonction coordonnée sur~$U$ qui est croissante vers~$x$
(c'est-à-dire $\lim_{x} \abs t>\lim_y \abs t)$.
Il existe une fonction analytique~$f$ sur~$U$ telle que $\omega|_U=f\, \da t/t$,
et $f$ n'a qu'un nombre fini de zéros et de pôles sur~$U$
car $\omega$ n'a qu'un nombre fini de zéros et de pôles au voisinage
de la partie compacte~$\overline U=U\cup\{x;y\}$.
On peut dès lors, quitte à prendre pour~$y$ un point plus proche de~$x$,
supposer que $f$ est inversible sur~$U$.
Dans ce cas, il existe un entier~$n$ tel 
que $\abs{f}=\abs t^n$ sur~$U$.
Alors, $\norm{\omega}=\abs t^n$ sur~$\mathopen]x;y\mathclose[$.
L'entier~$n$ ne dépend pas du choix de~$y$,
donc ne dépend que du germe en~$x$ de la cellule~$e$;
on le note~$n_e(\omega)$; 
c'est la pente de~$\omega$ en~$x$ le long de~$e$.
(Elle est positive ou nulle si $\norm \omega \leq \norm\omega(x)$
au voisinage de~$x$ sur~$e$.)

\subsection{}\label{ss-parametres-moderes}
Soit $x\in X$ un point de type~II ou~III
et soit $(e_i)$ une famille finie de germes de $1$-cellules 
dont $x$ est une extrémité.

Il existe une famille finie $(y_i)$, où $y_i\in e_i$ pour tout~$i$,
une famille finie $(U_i)$ d'ouverts de~$X$,
chaque~$U_i$ étant isomorphe à une couronne ouverte de squelette
l'intervalle~$\mathopen]x;y_i\mathclose[$,
et une fonction analytique~$t$ au voisinage de~$x$,
telles que pour tout~$i$, $t|_{U_i}$ soit une fonction
coordonnée sur la couronne~$U_i$, croissante vers~$x$,
et qui induit un paramètre modéré monomial en~$x$.

En effet, supposons tout d'abord que $x$ est de type II. 
Il suffit alors
de prendre une fonction analytique~$t$
de valeur absolue~$1$ en~$x$ telle que $\widetilde{t(x)}$
ait un zéro d'ordre exactement~$1$ 
en chacun des points fermés de la courbe~$C_x$
correspondant aux cellules~$e_i$ issues de~$x$ dans~$\mathscr C$,
et un unique pôle en un point fermé auxiliaire, 
choisi d'ordre premier à~$p$ si $\tilde k$ est de caractéristique~$p>0$.

Supposons maintenant $x$ de type III.
Il possède alors un voisinage ouvert~ $U$ qui est une couronne
ouverte telle que $U\cap e_i$ soit contenu dans le squelette de~$U$.
En particulier, les~$e_i$ ne définissent au plus que deux
germes de cellule en~$x$.
Toute fonction coordonnée sur~$U$ est un paramètre modéré et monomial en~$x$.

On dira que $t$ est un paramètre adapté aux cellules~$e_i$.

\subsection{}
\def\ord{\operatorname{ord}\nolimits}
Soit $x\in X$ un point de type~II.
Comme $k$ est algébriquement clos,
il existe $a\in k^\times$ tel que $\norm\omega(x)=\abs a$.
Alors $\omega/a$ possède une réduction au sens de Temkin,
qui est une forme différentielle
\emph{méromorphe} sur la courbe~$C_x$ que nous noterons $\widetilde{\omega(x)}$.
Elle est bien définie à multiplication près 
par un élément de~$\widetilde k^\times$.

Par ailleurs, chaque $1$-cellule~$e$
dont~$x$ est une extrémité correspond à un
point fermé $P$ de la courbe~$C_x$; 
on notera~$\ord_e(\widetilde{\omega(x)})$ 
le coefficient de $P$ dans le diviseur de $\widetilde {\omega(x)}$ ; cet entier relatif ne dépend pas du choix de l'élément~$a$.

\begin{lemm}\label{lemme-ne-ordre}
Soit $k$ un corps valué complet, non trivialement valué, algébriquement clos
et soit $X$ une courbe $k$-analytique lisse.
Soit $\omega$ une forme différentielle méromorphe régulière sur~$X$.
Soit $x\in X$ un point de type~II et soit $e\subset X$ 
une $1$-cellule dont $x$ est une extrémité.
On a $n_e(\omega) = 1 + \ord_e(\widetilde{\omega(x)})$.
\end{lemm}
\begin{proof}
L'assertion résulte essentiellement du théorème~4.3.3~(ii)
de~\cite{CohenTemkinTrushin-2016}. Donnons néanmoins quelques précisions.
Quitte à multiplier~$\omega$ par un scalaire, on suppose
que $\norm{\omega(x)}=1$. Soit $t$ une fonction analytique au voisinage
de~$x$ qui fournit un paramètre en~$x$, adapté à la cellule~$e$;
soit $U$ un ouvert de~$X$ qui est une couronne ouverte
dont le squelette est un intervalle $\mathopen]x;y\mathclose[$ de~$e$n et sur laquelle
$\omega$ n'a ni zéro ni pôle.
Il existe une fonction méromorphe régulière régulière~$\phi$
au voisinage de~$x$ telle que $\omega = \phi\cdot \da t/t$ au voisinage de~$x$.

Quitte à multiplier~$t$ par un scalaire, on peut
supposer que $\abs{t(x)}$; on a alors $\abs{\phi(x)}=1$.

Posons $\tau=\widetilde{t(x)}\in\hrt x^1$;
il a un zéro simple au point~$v$ de~$C_x$ qui correspond à la branche~$e$.
Par construction, $\da t/t$ se réduit en $\da \tau/\tau$,
et $\widetilde{\omega(x)}=\widetilde{\phi(x)} \da \tau/\tau$.

On a alors $\abs{\phi(t)} = \abs {t}^{n_e(\omega)}$ sur~$U$,
si bien que $\widetilde{\phi(x)}$ est d'ordre $n_e$ en $v$.  Il s'ensuit que 
$\ord_e(\widetilde{\omega(x)})=n_e-1$, d'où le lemme.
\end{proof}

\begin{lemm}\label{lemme-ne-degdisque}
Soit $x$ un point de type~II ou~III de~$X$
et soit $D$ une partie ouverte de~$X$ qui est isomorphe à un disque
ouvert et dont le bord topologique
dans $X$ est égal à~$\{x\}$.
Soit $e\subset X$ une cellule dont $x$ est une extrémité
et telle que $e\setminus\{x\}\subset D$.
La forme~$\omega$ n'a qu'un nombre fini de zéros et de pôles sur~$D$,
et l'on a $n_e(\omega) = 1 + \deg(\omega|_D)$.
\end{lemm}
\begin{proof}
Comme $\overline D=D\cup\{x\}$, c'est une partie compacte de~$X$,
de sorte que $\omega$ n'a qu'un nombre fini de zéros et de pôles sur~$D$.

Soit $t\in\mathscr O(D)$
une coordonnée sur ce disque (qui induit un isomorphisme de~$D$
sur un disque ouvert). Comme $\da t$ est une base de~$\Omega_D$,
il existe une unique fonction analytique $f$ sur~$D$
telle que $\omega = f\,\da t$.
Comme le diviseur de~$\omega|_D$ est fini, 
il existe une fonction inversible~$u$ sur~$D$ 
et une fonction rationnelle $g\in k(t)$ telle que $f=u g$.
La fonction $\abs u$ est constante sur~$D$.
Pour $y\in D$ assez proche de~$x$, on a donc $\abs {f(y)}= \abs{u(0)} 
  \cdot  \abs{t(y)}^d$, 
où $d$ est 
le nombre de zéros et de pôles de
$g$ dans $D$, comptés avec leur multiplicités,
c'est-à-dire $\deg(\omega|_D)$.
Par suite, \[\norm{\omega(y)} = \abs{f(y)}\cdot \abs{t(y)} \cdot\left\|\frac{\da t}t(y)\right\|=\abs{u(0)}\cdot \abs{t(y)}^{1+d}\cdot
\left\|\frac{\da t}t(y)\right\|.\]
Si $t(y)$ appartient au squelette canonique de~$\gm$,
on a $\norm{(\da t/t}(y))=1$ et donc
$\norm{\omega}(y)=\abs{u(0)}\cdot\abs{t(y)}^{1+d}$.
Autrement dit, $n_e(\omega)=1+d$, comme il fallait démontrer.
\end{proof}

\begin{defi}
Soit $k$ un corps ultramétrique complet algébriquement clos 
et non trivialement valué, 
soit~$X$ une courbe $k$-analytique lisse 
et soit $f$ une fonction tropicalement continue à support compact sur~$X$. 
Nous dirons qu'une partie 
squelettique $\Sigma$ de $X$ est \emph{$f$-admissible}
si elle satisfait les conditions suivantes : 

\begin{itemize}
\item $\Sigma$ est compacte ; 
\item $f$ est localement constante sur $X\setminus \Sigma$ ;
\item toute composante connexe $U$ de $X\setminus \Sigma$ ou bien 
est contenue dans le lieu d'annulation de $f$, ou bien est un
disque ouvert relativement compact dans $X$.
\end{itemize}

Si $F$ est un ensemble de fonctions tropicalement continues à support compact, 
nous dirons que $\Sigma$ est $F$-admissible s'il est $f$-admissible pour toute $f\in F$. 
\end{defi}

\begin{exem}\label{exemple-f-admissible}
Soit $k$ et $X$ comme ci-dessus, soit $\Gamma$ une partie
squelettique compacte de $X$ et soit $F$ un ensemble fini de fonctions
tropicalement continues à support compact sur $X$.
Soit $V$ un domaine
analytique compact de $X$ contenant $\Gamma$ et le support de $f$ pour toute $f\in F$. 
Soit $(V_i)$ un recouvrement fini de $V$ par des domaines analytiques compacts
tels qu'il existe pour tout $i$ un moment $g_i$ sur $V_i$ tropicalisant $f|_{V_i}$
pour toute $f\in F$. 
Soit $\Delta$ la réunion de $\Gamma$, des $\Sigma_{g_i}$, de $\partial V$, 
de l'ensemble des points $x$ de $V$ tels que $g(x)>0$, et de l'ensemble des boucles tracées sur $V$ ; soit $\Delta'$
la partie squelettique obtenue à partir de $\Delta$ en lui adjoignant, pour toute composante connexe 
$W$ de $V$ telle que $\Delta\cap W=\emptyset$, un point de type II ou III arbitrairement choisi sur $W$. 

Si $C$ et $D$ sont deux composantes connexes distinctes de $\Delta'$, il existe au plus un segment $[x,y]$ tracé sur $V$ avec $x\in C, y\in D$ et $]x,y[\subset V\setminus \Delta'$ : 
en effet s'il
existait deux  tels segment $[x,y]$
 et $[x',y']$ distincts, 
la concaténation d'un segment $[x,x']$ tracé
sur $C$, de $[x',y']$, d'un segment $[y',y]$ tracé sur $D$ et de $[y,x]$ contiendrait 
nécessairement une boucle non incluse dans $\Delta'$. Notons $I_{C,D}$ cet unique segment
lorsqu'il existe, et désignons par $\Sigma$ la réunion de $\Delta'$ et des différents
$I_{C,D}$ pour $C$ et $D$ variables. Par construction, $\Sigma$ est une partie squelettique
de $V$ qui contient
$\Gamma, \bigcup \Sigma_{g_i}$, $\partial V$ tous les point $x$ de $V$ tels que
$g(x)>0$, et toutes les boucles de $V$, et dont l'intersection 
avec toute composante connexe de $V$ est non vide et connexe. 

Si $W$ est une composante connexe de $V\setminus \Sigma$ alors $W$ est un arbre
(puisqu'elle ne contient aucune boucle) et n'est pas compacte 
(sinon ce serait une composante connexe de $V$ ne rencontrant pas $\Sigma$). 
Le bord topologique de $W$ dans $V$ est alors un singleton : il est en effet non vide car $W$
n'est pas compacte, et s'il contenait deux points distincts $x$ et $y$ ceux-ci seraient
deux points de $\Sigma$ situés sur la même composante connexe de $V$ et seraient donc à ce titre 
joints par un segment tracé sur $\Sigma$, mais ils le seraient aussi par un intervalle ouvert 
tracé sur l'arbre $W$, ce qui fournirait une boucle de $V$ non située sur $\Sigma$. 
Ainsi $W$ est un arbre à un bout relativement compact sur $V$, qui n'en rencontre pas le bord
et ne contient aucun point de genre strictement positif. C'est en conséquence un disque ouvert. 

Soit maintenant $U$ une composante connexe de $X\setminus \Sigma$. Supposons que $U$
rencontre $V$. Comme $\partial V\subset \Sigma$, on a $U\cap V=U\cap \int(V)$, et $U\cap V$
apparaît ainsi comme une partie non vide, ouverte et fermée de $U$, qui est donc égale
à $U$ par connexité. Autrement dit, la composante $U$ est contenue dans $V$, et c'est dès lors
une composante connexe de $V\setminus \Sigma$. C'est en conséquence par ce qui précède un disque
ouvert relativement compact, et $f$ est constante sur $U$ pour toute $f\in F$ puisque
$\Sigma$ contient
$\bigcup_i \Sigma_{g_i}$. Et si $U$ ne rencontre pas $V$ alors $f|_U$ est identiquement nulle
pour toute $f\in F$ puisque $V$ contient le support de $f$. On en conclut que $\Sigma$ est $F$-admissible.
\end{exem}

\begin{lemm}\label{lemme-courant-canonique}
Soit $k$ et $X$ comme dans la définition ci-dessus, et soit 
$F$ un ensemble fini de fonctions tropicalement continues à support compact sur $X$. 
\begin{enumerate}
\item Pour toute partie squelettique compacte $\Gamma$ de $X$, 
il existe une partie squelettique de $X$ qui est $F$-admissible
et contient $\Gamma$ ; 
\item Soit $\Sigma$ une partie squelettique $F$-admissible de $X$ et soit $\Tau$ une partie 
squelettique compacte de $X$ contenant $\Sigma$ ; 
les assertions suivantes sont équivalentes : 
\begin{itemize}
\item[(i)] la partie squelettique $\Tau$ est $F$-admissible ;
\item[(ii)] toute composante connexe de $\Tau$ qui ne rencontre pas $\Sigma$ est contenue dans le lieu d'annulation 
de $f$ pour chaque $f\in F$ ; 
\end{itemize}
\item Soit $\Sigma$ et $\Tau$ comme dans b) telles que les conditions
équivalentes (i) et (ii) soient satisfaites ; on a pour toute $f\in F$ l'égalité 
\[\langle K_\Tau, f|_\Tau\rangle=\langle K_\Sigma, f|_\Sigma\rangle.\]
\item Soit $\Sigma$ une partie $F$-admissible de $X$ et soit $U$ un voisinage ouvert de $\Sigma$ contenant
la réunion des supports des éléments de $F$. Alors $\Sigma$ est une partie $(F|_U)$-admissible de $U$. 
\end{enumerate}
\end{lemm}

\begin{proof}
L'exemple \ref{exemple-f-admissible}
fournit un exemple explicite de partie squelettique $F$-admissible de
$X$ contenant $\Gamma$, d'où \emph a). 

Prouvons \emph b). Supposons tout d'abord que $\Tau$ est $F$-admissible. Soit $\Theta$ une composante connexe de $\Tau$ ne rencontrant
pas $\Sigma$ et soit $f\in F$ ; montrons que $f|_{\Theta}=0$. 
La partie squelettique compacte $\Theta$ est contenue dans une composante connexe $U$ de $X\setminus \Sigma$ ; il suffit de
vérifier que la valeur constante de $f$ sur $U$ est nulle. On suppose par l'absurde que ce n'est pas le cas. Dans ce cas $U$ est un disque ouvert, 
et son bord contient un unique point $x$ situé sur $\Sigma$. Soit $y$ l'image de $x$ sur $\Theta$ par la rétraction canonique de l'arbre compact
$U\cup\{x\}$ sur son sous-arbre compact non vide $\Theta$. La composante connexe $V$ de $]y,x[$ dans $U\setminus \Theta$ est alors une
composante connexe de $X\setminus \Tau$, et donc un disque ouvert relativement compact puisque $\Tau$
est $f$-admissible et que la valeur constante de $f$ sur $V$ n'est pas nulle. Le bord topologique de $V$ dans $X$ est dès lors un singleton, ce qui 
absurde puisque ce bord est égal à $\{x,y\}$ par construction. 

Réciproquement, supposons que (ii) est satisfaite. Comme $\Tau\supset \Sigma$, toute 
fonction $f$ appartenant à $F$ est localement constante sur 
$X\setminus \Tau$. Soit $U$ une composante connexe de $X\setminus \Tau$.
Supposons qu'il existe $f\in F$ dont la valeur constante sur $U$
est non nulle, et soit $V$ la composante connexe de $X\setminus \Sigma$ 
contenant $U$. Nous allons montrer que $V$ est un disque
ouvert relativement compact dans $X$, ce qui prouvera (i). 
Par construction, $V$ est un disque ouvert dont l'unique point $x$
du bord est situé sur $\Sigma$. 
Soit $E$ une composante connexe de $\Tau \cap (V\cup\{x\})$. 
Elle est compacte, et si elle ne contient pas $x$ c'est une composante connexe de
$\Tau$ ne rencontrant pas $\Sigma$, ce qui est absurde au vu de l'hypothèse (ii)
et puisque la valeur constante de $f$ sur $V$ est non nulle. 
Par conséquent, $\Tau\cap  (V\cup\{x\})$ est un arbre compact contenant $x$. On distingue 
dès lors deux cas. Si $\Tau\cap V$ est vide alors $U=V$, et $U$ est donc un disque ouvert
relativement compact de $X$. Et si $\Tau\cap V$ est non vide, c'est un sous-arbre fermé de $V$
contenant un intervalle ouvert aboutissant à $x$ ; en conséquence $V\setminus \Tau$ est une union
disjointe $\coprod V_i$ de disques ouverts relativement compacts dans $V$. Les $V_i$ sont alors
à la fois ouverts et fermés dans $X\setminus \Tau$, si bien que ce sont des composantes
connexes de $X\setminus \Tau$, et $U$ est dès lors nécessairement l'un des $V_i$, 
ce qui termine la preuve de \emph b). 

Montrons enfin \emph c). Soit
$f\in F$. Soit $\Tau_1,\ldots, \Tau_m$ les composantes connexes de
$\Tau\setminus \Sigma$. Pour tout $i$, soit $U_i$ la composante connexe de 
$X\setminus \Sigma$ qui contient $\Tau_i$. Soit $I$ l'ensemble des indices
tels que la valeur constante $a_i$ de $f$ sur $U_i$ soit non nulle. 
Si $i\notin I$, aucun point de $\overline {\Tau_i}$ ne contribue 
à $\langle K_\Tau, f|_\Tau\rangle$
ni à $\langle K_\Sigma, f|_\Sigma\rangle$. Si
$i\in I$ alors 
$U_i$ est un disque ouvert 
relativement compact dans $X$, et $\Tau_i$ n'est pas compacte (sans quoi
ce serait une composante connexe de $\Tau$, ce qui contredirait (ii)). L'adhérence de 
$\Tau_i$ dans $X$ est donc égale à $\Tau_i\cup\{x_i\}$, où $x_i$ est l'unique point du bord
de $U_i$ (remarquons que si $i$ et $j$ sont
deux indices différents de $I$ alors $U_i$ et $U_j$ sont disjoints, sans quoi 
$\Tau_i$ et $\Tau_j$ coïncideraient à l'approche de $x_i$).
La valence de $\Tau$ en un point $x$ de $\Sigma$ en lequel $f$ ne
s'annule pas est donc égale
à la somme de la valence de $\Sigma$ en $x$ et du cardinal de $\{in I,x_i=x\}$. 
Il s'ensuit que
\[\langle K_\Tau, f|_\Tau\rangle=\langle K_\Sigma, f|_\Sigma\rangle
+\sum_{i\in I} a_i\left(1+\sum_{x\in S_i}(v(x)-2)\right),\]
où $S_i$
désigne l'ensemble des sommets de $\Tau_i$, où
la valence $v(x)$ est calculée relativement à $\Tau_i$, et où l'on a utilisé
le fait que $g(x)=0$ pour tout $x\in S_i$ puisque $U_i$ est un disque. 
Pour tout $i$, notons $\Tau'_i$ l'arbre compact 
$\Tau_i\cup\{x_i\},\,S'_i=S_i\cup\{x_i\}$ l'ensemble de ses sommets
et $E_i$ l'ensemble de ses arêtes. 
On a alors
\begin{align*}
1+\sum_{x\in S_i}(v(x)-2)&=\sum_{x\in S'_i}v(x)-2\Card(S_i)\\
&=2\Card(E_i)-2\Card(S'_i)+2\\
&=2(1-\chi(\Tau'_i))\\
&=0,
\end{align*}
si bien que 
$\langle K_\Tau, f|_\Tau\rangle=\langle K_\Sigma, f|_\Sigma\rangle$.

Montrons enfin (d). Comme $\Sigma$ est $F$-admissible, $\Sigma$ est une partie squelettique compacte de $U$ et $f|_U$
est localement constante sur $U\setminus \Sigma$. Soit maintenant $V$ une composante connexe de $U\setminus \Sigma$
telle qu'il existe une fonction $f\in F$ dont la valeur constante sur $V$ est non nulle. 
Soit $V'$ la composante connexe de $X\setminus \Sigma$ qui contient $V$. La valeur constante de $f$ sur $V'$ est alors non nulle, 
ce qui entraîne que $V'$ est un disque ouvert relativement compact dans $X$, dont l'unique point du bord $x$ est situé sur $\Sigma$ ; et
cela entraîne aussi que $V'$ est contenu dans le support de $f$, et a fortiori dans $U$. Il vient $V=V'$, si bien que $V$
est un disque ouvert dont l'adhérence dans $U$ est le compact $V\cup\{x\}$ (puisque $\Sigma\subset U$). 
\end{proof}

\subsection{}
Si $F$ est un ensemble fini de fonctions tropicalement continues à support compact, il résulte 
du lemme \ref{lemme-courant-canonique} que l'ensemble $\mathscr S$ 
des parties squelettiques $F$-admissibles de $X$ est non vide, et même filtrant, 
et que si $\Sigma\in \mathscr S$ et $f\in F$ le nombre réel 
$\langle K_\Sigma, f|_\Sigma\rangle$ ne dépend pas de $\Sigma$. 

Il est donc licite de définir pour toute fonction tropicalement continue $f$ à support
compact
$\langle K^\sq_X,f\rangle$ comme étant égal à $\langle K_\Sigma,f|_\Sigma\rangle$
pour n'importe quelle partie squelettique $f$-admissible $\Sigma$ de $X$. 
Si $g$ est une autre fonction tropicalement continue à support compact on montre en utilisant 
une partie squelettique $\{f,g\}$-admissible que 
\[\langle K^\sq_X,f+\lambda g\rangle=\langle K^\sq_X,f\rangle+\lambda 
\langle K^\sq_X,g\rangle\] pour tout $\lambda\in \R$. Ainsi $K_X^\sq$
apparaît-il comme une forme linéaire sur l'espace des applications tropicalement continues
à support compact sur $X$. En le restreignant à $\mathscr A_{\mathrm c}(X)$, on obtient
un courant d'ordre $0$ sur $X$. 

\begin{theo}
Soit $k$ un corps valué complet, non trivialement valué, algébriquement clos
et soit $X$ une courbe analytique lisse sur~$k$.
Soit $c_1(X)$ le courant de courbure du fibré canonique de~$X$,
muni de sa semi-métrique de Temkin.
Soit $f$ une fonction lisse et à support compact sur~$X$.
On a 
\[\langle c_1(X),f\rangle=\langle K^\sq_X, f\rangle.\]
\end{theo}
\begin{proof}
Soit $V$ un domaine analytique compact de $X$ contenant le support de $f$ et soit $\Sigma$
une partie squelettique $f$-admissible de $X$ contenue dans $V$ et contenant 
$\partial V$ (il en existe par l'exemple \ref{exemple-f-admissible}). 
Il s'agit de démontrer que
\[\langle c_1(X),f\rangle=\langle K_\Sigma, f|_\Sigma\rangle.\]
Il existe un voisinage ouvert $U$ de $V$
dans $X$ qui est de la forme $V\cup\bigcup_i X_i$ où chaque 
$X_i$ est une couronne ouverte contenue dans $X\setminus V$ dont le bord est une paire
$\{x_i,y_i\}$ avec $x_i\in \partial V$ et $y_i\notin V$, et où les $X_i$ sont deux
à deux disjointes. En prolongeant chacune des $X_i$ en un disque ouvert de bord $\{x_i\}$ 
on obtient une courbe $k$-analytique propre et donc projective lisse $Y$ munie d'une 
immersion ouverte $U\hookrightarrow Y$ ; comme $Y$ est une courbe projective, le fibré en droites
$\Omega_{Y/k}$ possède une section méromorphe régulière $\omega$, dont la restriction à $U$ est une section méromorphe
régulière de $\Omega_{U/k}$. Les deux termes de l'égalité à démontrer  $\langle c_1(X),f\rangle=\langle K_\Sigma, f|_\Sigma\rangle$
peuvent être calculés en remplaçant $X$ par $U$ puisque $U$ contient $\Sigma$ et le support de $f$
et que $\Sigma$ est encore $f|_U$-admissible d'après 
le lemme \ref{lemme-courant-canonique} (d). On peut donc supposer qu'il existe une forme différentielle méromorphe
régulière $\omega$ sur $X$.

Le théorème~\ref{theo.poincare-lelong-metriques} affirme
alors que $c_1(X)=\ddc [\log\norm\omega^{-1}] + \delta_{\div(\omega)}$.
Par suite,  on a
\[ \langle c_1(X), f\rangle = \int_X \log \norm{\omega}^{-1} \ddc f
 + \langle \delta_{\div(\omega)}, f\rangle. \]

On commence par calculer le premier terme.
Soit $\mathscr C$ une triangulation de~$\Sigma$.
Puisque la fonction lisse $f$ est localement constante en dehors de
$\Sigma$, ce dernier est un support fort de~$\ddc f$, si bien que 
\[ \int_X \log\norm\omega^{-1}\ddc f = 
 \sum_{e\in\mathscr C_1} \int_e \log \norm{\omega^{-1}}\ddc f. \]

Pour toute arête~$e\in\mathscr C_1$,
soit $U_e$ la couronne ouverte de $X$ 
dont le squelette 
est la cellule ouverte~$\mathring e$.
Choisissons une coordonnée~$t_e$ sur cette couronne;
la fonction $\tau_e=\log\abs{t_e}$ identifie~$\mathring e$ 
à un intervalle~$\mathopen]a_e;b_e\mathclose[$ de~$\R$,
de manière compatible aux calibrages. (Par convention, on a choisi
$a_e<b_e$.)

Pour tout sommet~$x\in\mathscr C_0$,
il existe un paramètre $t_x$ adapté 
à toutes les arêtes aboutissant à $x$ (\ref{ss-parametres-moderes}). On peut dès lors raffiner 
la triangulation~$\mathscr C$
de sorte que, pour chaque arête~$e$, le paramètre~$t_e$
choisi plus haut se prolonge en un paramètre
encore noté $t_e$ sur un voisinage analytique de $\overline{U_e}$ ; et l'on peut aussi faire en sorte
que $\omega_e|_{U_e}$ n'ait ni zéro ni pôle, et que $f|_e$ soit égale à 
$f_e\circ \tau_e$ pour une certaine fonction indéfiniment dérivable $f_e$ sur le segment $[a_e;b_e]$. 

Notons que le choix des paramètres~$t_e$
fournit une orientation du graphe~$\Sigma$: l'arête~$e$
est orientée de sorte que la fonction~$\tau_e$ y soit croissante.
Une arête~$e$ contenant un point~$x\in\mathscr C_e$
sera ainsi dite sortante si $a_e=\tau_e(x)$, et entrante si $b_e=\tau_e(x)$.

Soit $e\in \mathscr C_1$. 
Puisque $\da t_e/t_e$ est une base de~$\Omega^1_{U_e/k}$,
c'est une une forme différentielle
méromorphe régulière au voisinage 
de~$\overline{U_e}$, et il existe dès lors 
une fonction méromorphe régulière~$\phi_e$ 
au voisinage de~$\overline{U_e}$
telle que $\omega$ s'écrive $\phi_e \,\da t_e/t_e$
sur le voisinage en question. Et comme $\omega|_{U_e}$ n'a ni zéro ni pôle,
la fonction $\phi_e|_{U_e}$ est inversible. 
Alors, pour tout~$x\in \mathring e$,
on a 
\[\norm{\omega(x)}=\abs{\phi_e(x)} \cdot \norm{\da t_e/t_e(x)}=\abs{\phi_e(x)}.\]
Comme $\phi_e$ est 
inversible sur $U_e$ il existe une fonction affine~$\psi_e$ sur l'intervalle~$[a_e;b_e]$
telle que $\psi_e(\tau_e(x))=\log\abs{\phi_e(x)}$ pour tout $x\in e$ ; notons
$n_e$ la pente de $\psi_e$ (c'est un entier relatif).
S
On a alors
\[ \int_{e} \log \norm{\omega}^{-1}\ddc f 
 = - \int_{a_e}^{b_e} \psi_e(\tau) f_e''(\tau) \,d\tau. \]

En intégrant par parties, il vient 
\begin{align*}
\int_{e} \log \norm{\omega}^{-1}\ddc f   
    & = - \left[ \psi_e(\tau)f_e'(\tau)\right]_{a_e}^{b_e}
     + \int_{a_e}^{b_e} n_e f_e'(\tau)\, d\tau \\
  & = \psi_e(a_e) f_e'(a_e)-\psi_e(b_e) f_e'(b_e) 
 + n_e (f_e(b_e)-f_e(a_e)).
\end{align*}
Alors, 
\begin{align*}
\int_U \log \norm{\omega}^{-1}\ddc f
 & = \sum_{e\in\mathscr C_1}
    \left( \psi_e(a_e) f_e'(a_e)- \psi_e(b_e) f_e'(b_e)
     + n_e (f_e(b_e)-f_e(a_e))\right) \\
& = \sum_{x\in \mathscr C_0}
 \sum_{\text{$e$ est sortante en~$x$}}
   \left( \log\abs{\phi_e(x)}  \cdot f_e'(\tau_e(x)) - n_e f(x)\right) \\
& \qquad\qquad {}
 - \sum_{\text{$e$ est entrante en~$x$}}
   \left( \log\abs{\phi_e(x)} \cdot f_e'(\tau_e(x)) - n_e f(x)\right)\\
& = \sum_{x\in\mathscr C_0} \sum_{e\ni x} \eps_e(x) 
   \left( \log\abs{\phi_e(x)} \cdot f_e'(\tau_e(x)) - n_e f(x)\right) ,
\end{align*}
où $\eps_e(x)$ vaut~$1$ si l'arête~$e$ est sortante en~$x$, et~$-1$ sinon.
Par suite,
\[ \int_U \log \norm{\omega}^{-1}\ddc f
 = \sum_{x\in\mathscr C_0} \log\norm{\omega(x)} \left(\sum_{e\ni x} \eps_e(x) f_e'(\tau_e(x))\right) - f(x) \left(\sum_{e\ni x} \eps_e(x) n_e\right).\]

\begin{lemm}
Pour tout $x\in\mathscr C_0$, on a 
$\sum_{e\ni x} \eps_e(x) f'_e(\tau_e(x))=0$.
\end{lemm}
\begin{proof}
Soit $g$ une fonction lisse sur~$X$, égale à~$1$ en~$x$ et
dont le support est compact et ne rencontre
aucune arête de~$\mathscr C$ autre que celles contenant~$x$.
La forme lisse $g\dc f$ sur~$X$ est de type~$(0,1)$ et à support
compact. 
La réunion des arêtes contenant~$x$ est un support fort de la forme~$g\dc f$;
par suite, son intégrale de bord vaut
\[ \int^\partial_X g\dc f = \sum_{e\ni x} \int^\partial_e g\dc f.\]
Or, 
pour toute arête~$e$, on a 
\[ \int^\partial_eg\dc f = g(a_e) f'_e(a_e) - g(b_e) f'_e(b_e)
= - \eps_e(x) f'_e(\tau_e(x)), \]
si bien que
\[ \int^\partial_X g\dc f = - \sum_{e\ni x} \eps_e(x) f'_e(\tau_e(x)).\]
Comme $X$ est sans bord, on a $\int^\partial_X g\dc f= 0$,
d'où le lemme.
\end{proof}

Ainsi,
\[
\int_X \log \norm\omega^{-1}\ddc f 
 = \sum_{x\in\mathscr C_0} f(x) \left( \sum_{e\ni x} \eps_e n_e \right).
\]
Fixons un point~$x\in\mathscr C_0$ ; il est de type~II.
Compte-tenu des conventions d'orientations, on a $\eps_e n_e = n_e(\omega)$
pour toute arête~$e\in\mathscr C_1$ qui contient~$x$.
Par suite,
\[ \sum_{e\ni x} \eps_e n_e = \sum_{e\ni x} (1 + \ord_e (\widetilde{\omega(x)}))
 = v(x) + \sum_{e\ni x} \ord_e (\widetilde{\omega(x)}). \]

Supposons que $f(x)\neq 0$. 
Dans ce cas la valeur constante de $f$ sur toute composante
connexe de $X\setminus \Sigma$ dont le bord contient $x$ est non nulle, 
si bien qu'une telle composante est un disque ouvert. 
Soit $\mathscr D_x$ l'ensemble des composantes
connexes de $X\setminus \Sigma$ qui sont des disques
ouverts de bord $\{x\}$, et soit $U_x$
le domaine analytique compact $\{x\}\cup \bigcup_{D\in \mathscr D_x}D$
de $X$. 
Le degré de $\omega|_{U_x}$ est la somme (presque nulle)
des degrés $\deg(\omega|_D)$, 
lorsque $D$ parcourt $\mathscr D_x$. 
Soit $D\in \mathscr D_x$. Le disque $D$ définit un germe 
d'arête $e$ aboutissant à $x$
et il résulte des lemmes 
\ref{lemme-ne-ordre} et \ref{lemme-ne-degdisque}
que 
$\ord_e(\widetilde{\omega(x)}) = n_e(\omega)-1=\deg(\omega|_D)$.
La section méromorphe $\widetilde{\omega(x)}$ de $\Omega^1_{C_x/\widetilde k^1}$ est de degré 
$2g(x)-2$. Supposons que $f(x)\neq 0$. Dans ce cas pour toute
composante connexe $U$ de $X\setminus \Sigma$ dont le bord contient $x$, la valeur constante de $f$ sur
$U$ est non nulle, ce qui force $U$ à être un disque ouvert de bord $\{x\}$, par $f$-admissibilité
de $\Sigma$. 
Tout point fermé de $C_x$ est donc associé ou bien à une arête de $\Sigma$, ou bien à l'un des 
disques $D\in \mathscr D_x$. 
Par conséquent, 
\[ \sum_{e \in \mathscr C_1, e\ni x} \ord_e(\widetilde{\omega(x)})
 = 2g(x)-2 - \sum_{D\in \mathscr D_x} \deg(\omega|_D) . \]
Il
vient 
\begin{align*}
\int_X \log \norm\omega^{-1}\ddc f &= \sum_{x\in \mathscr C_0} f(x)\left( \sum_{e\in \mathscr C_1, e\ni x} \eps_e n_e \right)\\
&=\sum_{x\in\mathscr C_0} f(x)\left(v(x)+
 \sum_{e\in \mathscr C_1, e\ni x} \ord_e (\widetilde{\omega(x)})\right)\\
 &=\sum_{x\in\mathscr C_0, f(x)\neq 0} f(x)\left(v(x)+
 \sum_{e\in \mathscr C_1, e\ni x} \ord_e (\widetilde{\omega(x)}\right)\\
 &=\sum_{x\in\mathscr C_0, f(x)\neq 0} f(x)
\left( 2g(x)-2 +v(x) - \deg(\omega|_{U_x})\right)\\
&=\sum_{x\in \mathscr C_0}
f(x)
\left( 2g(x)-2 +v(x)\right)-\sum_{x\in \mathscr C_0, f(x)\neq 0} f(x)\deg(\omega|_{U_x})\\
&=\langle K_X^\sq, f\rangle \sum_{x\in \mathscr C_0, f(x)\neq 0} -f(x)\deg(\omega|_{U_x}).
\end{align*}
Il suffit pour conclure de s'assurer que $\sum_{x\in \mathscr C_0, f(x)\neq 0} f(x)\deg(\omega|_{U_x})
=\langle \delta_{\div(\omega)},f\rangle.$
Soit $a$ un pont situé sur $\div(\omega)$ et soit $U$ la composante connexe de $X\setminus \Sigma$ 
contenant $a$. Comme $\omega$ est sans zéro ni pôle sur chacune des couronnes $U_e$, le bord de
$U$ est contenu dans $\mathscr C_0$. Si $f(a)=0$ alors $f$ est identiquement nulle sur $U$ et donc
en tout point de $\partial U$. Sinon la valeur constante de $f$ sur $U$ est non nulle et $U$ est dès lors un disque ouvert ; 
son bord est un singleton $\{x\}$ et $f(x)=f(a)\neq 0$. On a par conséquent
\begin{align*}
\langle \delta_{\div(\omega)},f\rangle&=\sum_{a\in \div(\omega)}f(a)\ord_a (\omega)\\
&=\sum_{a\in \div(\omega),f(a)\neq 0}f(a)\ord_a (\omega)\\
&=\sum_{x\in  C_0, f(x)\neq 0} f(x)\sum_{D\in \mathscr D_x}\deg(\omega|_D)\\
&=\sum_{x\in  C_0, f(x)\neq 0} f(x)\deg(\omega|_{U_x}).\qedhere
\end{align*}
\end{proof}

\section{Produits de courants de courbure}

Soit $X$ un $n$-espace $k$-analytique sans bord.

\subsection{}\label{ss.prod-c1-BT}
Considérons des fibrés en droites $\overline{L_1},\dots,\overline{L_m}$ 
sur~$X$ munis de métriques dpsh. On dispose ainsi de courants
de courbure $c_1(\overline{L_i})$, localement différences de courants
positifs fermés.

Les métriques dpsh sur les~$L_i$ fournissent des sections
globales~$u_i$ du faisceau $\RPSH_X/\log \abs{\mathscr O_X^\times}$ 
et donc aussi des sections globales du faisceau
$\RPSH_X/\Harm_X$.
On a expliqué au paragraphe~\ref{ss.bt-Rpsh} comment 
la construction $(v_1,\dots,v_m)\mapsto \ddc(v_1)\wedge \dots \wedge \ddc(v_p)$ sur $(\RPSH_X)^m$ passe au quotient 
en un morphisme $m$-linéaire
de $(\RPSH_X/\Harm_X)^m$ dans l'espace des courants de bidegré~$(m,m)$.
On définit alors le courant 
$c_1(\overline {L_1})\wedge \dots \wedge c_1(\overline{L_m})$
comme l'image de la suite $(u_1,\dots,u_m)$ par ce morphisme quotient.

Plus généralement, si $T$ est un courant positif fermé sur~$X$,
on définit de même 
$c_1(\overline {L_1})\wedge \dots \wedge c_1(\overline{L_m})\wedge T$.

Sur un ouvert~$U$ de~$X$ sur lequel chacun des~$L_i$
possède une section inversible~$s_i$, ces courants sont donnés par 
l'expression
\[ c_1(\overline {L_1})\wedge \dots \wedge c_1(\overline{L_m})\wedge T|_U 
=
 \ddc \log \norm{s_1}^{-1} \wedge \dots \wedge \ddc \log \norm{s_m}^{-1} \wedge T . \]
Il est localement $\ddc$-exact si $m\geq1$.
Sa formation est multilinéaire en les~$\overline{L_i}$.

Ce courant
$c_1(\overline {L_1})\wedge \dots \wedge c_1(\overline{L_m})\wedge T$
est:
\begin{itemize}
\item 
de Bedford--Taylor, si $T$ est est de Bedford--Taylor
(\S\ref{ss.bt-def});
\item
positif, 
si les métriques des~$\overline {L_i}$ sont psh
(\S\ref{ss.bt-def});
\item
squelettique
si les métriques des~$\overline{L_i}$ sont~\pl
et si le courant~$T$ est squelettique
(corollaire~\ref{coro.ddcuT-paralineaire}).
\end{itemize}

Si les métriques de~$L_i$ sont tropicalement d-convexes, 
le support du courant
$c_1(\overline {L_1})\wedge \dots \wedge c_1(\overline{L_m})\wedge T$
est localement contenu  dans une partie squelettique de~$X$
(remarque~\ref{rema.support-bt-squelette}).

\begin{prop}\label{prop.bt-berk-projection}
Soit $j\colon Y\to X$ un morphisme topologiquement propre 
entre espaces analytiques sans bord et topologiquement séparés,
et soit $S$ un courant positif fermé sur~$Y$.
Le courant $j_*S$ sur~$X$ est positif fermé et l'on a l'égalité
\[ c_1(\overline L_1) \wedge\cdots\wedge c_1(\overline L_m)\wedge j_*S
 = j_* 
(c_1(j^*\overline L_1) \wedge\cdots\wedge c_1(j^*\overline L_m)\wedge S). \]
\end{prop}
\begin{proof}
L'assertion est locale sur~$X$ ; par multilinéarité, 
on peut supposer que les métriques des fibrés considérées
sont~psh.
Compte tenu des définitions
l'assertion se déduit alors de la proposition~\ref{prop.bt-projection},
par récurrence sur~$m$.
\end{proof}

\begin{coro}
Soit $X$ un $n$-espace $k$-analytique sans bord.

\begin{enumerate}
\item Soit $K$ une extension complète de~$k$ et soit $p\colon X_K\to X$
le morphisme de changement de base.
On a 
\[ p_* (c_1(p^*\overline L_1) \wedge\cdots\wedge c_1(p^*\overline L_m))
    = c_1(\overline L_1) \wedge\cdots\wedge c_1(\overline L_m).
\]

\item
Soit $Y$ un $n$-espace $k$-analytique sans bord,
et soit $f\colon Y\to X$ un morphisme fini 
dont le degré calibré~$\deg_f$ est constant de valeur~$d$ sur~$X^\gen$.
On a 
\[ f_* (c_1(f^*\overline L_1) \wedge\cdots\wedge c_1(f^*\overline L_m))
    = d\cdot c_1(\overline L_1) \wedge\cdots\wedge c_1(\overline L_m).
\]
\end{enumerate}
\end{coro}
\begin{proof}
Compte tenu de la proposition précédente,
cela découle des propositions~\ref{prop.image-directe-fini}
et~\ref{prop.image-directe-chgt-base}.
\end{proof}

\begin{prop}
Soit $X$ un $n$-espace $k$-analytique sans bord.
Soit $K$ une extension valuée de~$k$, universellement multiplicative;
soit $\pi\colon X_K\to X$ le morphisme de changement de base
et soit $\sigma\colon X\to X_K$ sa section de Shilov.

Soit $\overline L_1,\dots,\overline L_p$ des fibrés
en droites  sur~$X$ munis de métriques d-psh, 
soit $T=c_1(\overline L_1) \wedge \cdots \wedge c_1(\overline L_p)$
et soit $T_K=\ddc(\pi^*\overline L_1)\wedge\dots\wedge\ddc(\pi^*\overline L_p)$.
Soit $\alpha$ une forme lisse sur~$X$ de bidegré~$(u,v)$,
où $\max(u,v)\geq n-p$.
Le support du courant $T_K \wedge \pi^*\alpha$
est contenu dans~$\sigma(X)$.
\end{prop}
\begin{proof}
Le résultat est local sur~$X$. On peut ainsi supposer que
les fibrés en droites considérés sont triviaux
et les courants de courbures $c_1(\overline L_m)$
de la forme $\ddc u_m$, où $u_m$ est une fonction d-psh sur~$X$.
L'assertion résulte alors de la proposition~\ref{rema.courants-chgt-base}.
\end{proof}

\begin{prop}\label{theo.formule-projection-c1}
Soit $p\colon Y\to X$ un morphisme d'espaces analytiques 
équidimensionnels, topologiquement séparés et sans bord;
posons $m=\dim(X)$, $n=\dim(Y)$  et $d=n-m$.
On suppose que le lieu de platitude de~$p$ est dense dans~$Y$.
Soit $\overline L_1,\dots,\overline L_r$
des fibrés en droites sur~$X$ munis de métriques
tropicalement d-convexes,
soit $\omega$ une forme G-lisse de type~$(d,d)$ sur~$Y$
dont le support est $p$-compact 
et soit $u_\omega\colon X\to\R$ la fonction 
définie par l'intégrale 
\[ u_\omega(x) = \int_{Y_x} \omega \]
de la forme de type~$(d,d)$ induite par~$\omega$ sur l'espace
$\hr x$-analytique~$Y_x$ si cet espace est de dimension~$d$,
et par $u_\omega(x)=0$ sinon.

Pour toute $(m-r,m-r)$-forme lisse~$\alpha$ à support compact sur~$X$,
\[ 
\langle c_1(p^*\overline L_1)\wedge\dots\wedge c_1(p^*\overline L_r),
\omega\wedge p^*\alpha\rangle
\]
est l'intégrale de la fonction~$u_\omega$
contre la mesure
$ c_1(\overline L_1)\wedge\dots\wedge c_1(\overline L_r)\wedge 
\alpha $ sur~$X$.
\end{prop}
\begin{proof}
En considérant une partition de l'unité sur~$X$, 
on se ramène au cas où les fibrés en droites~$L_m$ sont triviaux
et les courants de courbures $c_1(\overline L_i)$
sont différences de deux courants de 
la forme $\ddc u_i$, où $u_i$ est une fonction d-convexe sur~$X$.
Par multilinéarité, 
l'assertion résulte alors du théorème~\ref{theo.formule-projection}.
\end{proof}

\begin{coro}\label{coro.formule-projection-c1}
Avec les notations de la proposition, on a l'égalité de
mesures
\[ p_* (c_1(p^*\overline L_1)\wedge \dots\wedge c_1(p^*\overline L_r)
\wedge p^*\alpha \wedge\omega) = u_\omega \cdot c_1(\overline L_1)\wedge\dots \wedge c_1(\overline L_r)\wedge\alpha . \]
\end{coro}

\begin{prop}
Soit $X$ un $n$-espace analytique sans bord,
soit~$p$ un entier naturel
et soit $\overline L_1,\dots,\overline L_m$ des fibrés en droites
sur~$X$ munis de métriques d-psh.
Soit $A$ une partie fermée de~$X$
qui est tropicalement~$\mathscr H^p$-négligeable.
Pour toute forme $\alpha\in\mathscr A^{n-p,n-p}(X)$, 
l'ensemble~$A$ est négligeable pour la mesure $
c_1(\overline L_1)\wedge \dots \wedge c_1(\overline L_p) \wedge \alpha$ sur~$X$.
\end{prop}
\begin{proof}
L'assertion est locale sur~$X$. On peut donc supposer
que les fibrés en droites~$\overline L_m$ sont triviaux
et les courants de courbure $c_1(\overline L_m)$ de la forme
$\ddc u_m$, où $u_m$ est une fonction d-psh sur~$X$.
L'assertion résulte alors du corollaire~\ref{coro.support-prod-ddc}.
\end{proof}

\begin{lemm}
Soit $\overline{L_1},\dots,\overline{L_m}$
des fibrés en droites sur~$X$ munis de métriques tropicalement d-convexes.
Soit $U$ un ouvert de~$X$ tel que $d_k(x)<n$ pour pour tout
$x\in X\setminus U$,
et soit $h$ une fonction mesurable sur~$U$, localement bornée.
Le courant $ h  c_1(\overline {L_1}|_U) \wedge
\dots \wedge c_1(\overline{L_m}|_U)$ sur~$U$ 
se prolonge par zéro à~$X$.
\end{lemm}
\begin{proof}
Démontrons que pour toute forme lisse~$\alpha$ 
de bidegré~$(n-m,n-m)$ sur~$X$, à support compact,
le support du 
courant $c_1(\overline {L_1}|_U) \wedge \dots \wedge c_1(\overline{L_m}|_U) 
\wedge \alpha|_U$ sur~$U$ est compact.
Ce support est l'intersection avec~$U$
du support du courant $c_1(\overline {L_1}) \wedge \dots \wedge c_1(\overline{L_m}) \wedge \alpha$.
Comme les métriques de~$\overline {L_i}$ sont dpsh
et comme le support de~$\alpha$ est compact,
le support de ce dernier courant
est contenu dans une partie squelettique compacte~$\Sigma$ de~$X$
(\S\ref{ss.prod-c1-BT}).
On a $d_k(x)=n$ pour tout $x\in \Sigma$, si bien que $\Sigma$
est contenu dans~$U$.
\end{proof}

\begin{prop}\label{prop.PL+BT}
Soit $\overline L,\overline{L_1},\dots,\overline{L_m}$
des fibrés en droites sur~$X$ munis de métriques tropicalement d-convexes.
Soit~$s$ une section méromorphe régulière de~$L$ sur~$X$,
soit~$U$ son ouvert de définition.
Le courant 
$\log \norm s^{-1}  c_1(\overline {L_1}) \wedge
\dots \wedge c_1(\overline{L_m})$ sur~$U$
possède un prolongement par zéro à~$X$;
notons-le $[\log \norm s^{-1} c_1(\overline {L_1}) \wedge
\dots \wedge c_1(\overline{L_m})]$.
On a l'égalité 
\begin{multline*}
 \ddc [\log \norm s^{-1} c_1(\overline {L_1}) \wedge
\dots \wedge c_1(\overline{L_m}) ]
+ c_1(\overline {L_1}) \wedge
\dots \wedge c_1(\overline{L_m})  \wedge \delta_{\div(s)} \\
 = c_1(\overline L) \wedge c_1(\overline {L_1}) \wedge
\dots \wedge c_1(\overline{L_m}) \end{multline*}
de courants sur~$X$.
\end{prop}  
\begin{proof}
L'existence de ce prolongement par zéro résulte du lemme précédent.
Pour démontrer l'égalité de courants
indiquée, on peut raisonner localement sur~$X$; 
on suppose donc que chaque~$L_i$ possède une section
inversible~$s_i$; posons $u_i=\log \norm{s_i}^{-1}$.
Raisonnant encore localement,
un argument de multilinéarité et la définition
des fonctions tropicalement d-convexes
permettent de supposer que les~$u_i$ sont tropicalement convexes.
Il s'agit alors de démontrer l'égalité de courants
\begin{multline*}
 \ddc [\log \norm s^{-1} \ddc(u_1)  \wedge
\dots \wedge \ddc(u_m)  ]
+ \ddc(u_1) \wedge
\dots \wedge \ddc(u_m) \wedge \delta_{\div(s)} \\
 = c_1(\overline L) \wedge \ddc(u_1)  \wedge
\dots \wedge \ddc(u_m).
\end{multline*}
Commençons par traiter le cas où les~$u_i$ sont lisses.
Soit $\alpha$ une forme de bidegré~$(n-m,n-m)$ sur~$X$, à support compact.
Par définition, on a 
\begin{align*}
\langle \ddc [\log \norm s^{-1} \ddc(u_1)  \wedge
\dots \wedge \ddc(u_m)  ], \alpha \rangle \hskip -5cm \\
& = \langle \log \norm s^{-1} \ddc(u_1)  \wedge
\dots \wedge \ddc(u_m)  , \ddc\alpha|_U \rangle  \\
& = \langle \log \norm s^{-1} , \ddc(u_1)  \wedge
\dots \wedge \ddc(u_m)  \wedge   \ddc\alpha|_U \rangle  \\
& = \langle \log \norm s^{-1} , \ddc (\ddc(u_1)  \wedge
\dots \wedge \ddc(u_m)  \wedge  \alpha|_U)  \rangle  \\
& = \langle \ddc\log \norm s^{-1} , \ddc(u_1)  \wedge
\dots \wedge \ddc(u_m)  \wedge  \alpha|_U  \rangle \\
& = \langle -\delta_{\div(s)} + c_1(\overline L),  \ddc(u_1)  \wedge
\dots \wedge \ddc(u_m)  \wedge  \alpha|_U  \rangle ,
\end{align*}
d'après la formule de Poincaré--Lelong pour~$\overline L$
(théorème~\ref{theo.poincare-lelong-metriques}).
La formule voulue s'en déduit.

Traitons maintenant le cas général.
On peut encore raisonner localement sur~$X$ et supposer 
qu'il existe un moment~$f\colon X\to T$ sur~$X$,
un polyèdre compact~$P$ de~$T_\trop$  contenant~$f_\trop(X)$,
des fonctions convexes lisses $(v_i^k)$ sur~$P$
telles que pour tout~$i$,
la suite $(v_i^k)$ converge uniformément vers une fonction~$v_i$
sur~$P$ telle que $u_i=f^*v_i$; pour tout~$i$, notons $u_i^k=f^*_\trop v_i^k$.
Pour tout~$k$, on a la formule 
\begin{multline*}
 \ddc [\log \norm s^{-1} \ddc(u_1^k)  \wedge
\dots \wedge \ddc(u_m^k)  ]
+ \ddc(u_1^k) \wedge
\dots \wedge \ddc(u_m^k) \wedge \delta_{\div(s)} \\
 = c_1(\overline L) \wedge \ddc(u_1^k)  \wedge
\dots \wedge \ddc(u_m^k).
\end{multline*}
D'après la proposition~\ref{prop.BT-continu}
le second terme du membre de gauche et le membre de droite
convergent, comme courants, vers l'expression correspondante.
Il suffit de justifier que 
les courants $[\log \norm s^{-1} \ddc(u_1^k)  \wedge
\dots \wedge \ddc(u_m^k)  ]$
convergent vers $[\log \norm s^{-1} \ddc(u_1)  \wedge
\dots \wedge \ddc(u_m)  ]$.
Soit $\alpha$ une forme lisse à support compact sur~$X$, de
bidegré~$(n-m,n-m)$.
Soit $\Sigma$ une partie squelettique compacte de~$X$,
réunion  de polyèdres caractéristiques de moments obtenus 
en concaténant la restriction de~$f$
et des moments tropicalisant~$\alpha$
sur des domaines compacts de~$X$ dont les
intérieurs recouvrent le support de~$\alpha$.
Alors $\Sigma$ est un support fort des
formes $ \ddc(u_1^k)  \wedge
\dots \wedge \ddc(u_m^k) \wedge \alpha  $ 
(lemme~\ref{lemm.Sigmaf-support-fort}).
Comme $d_k(x)<n$ pour tout $x\in X\setminus U$, on a $\Sigma\subset U$.

Le support
des courants $ \ddc(u_1^k)  \wedge
\dots \wedge \ddc(u_m^k) \wedge \alpha  $ sur~$X$
est contenu dans~$\Sigma$, 
et ces courants convergent vers  le courant
$ \ddc(u_1^k)  \wedge
\dots \wedge \ddc(u_m^k) \wedge \alpha$,
d'après la proposition~\ref{prop.BT-continu}.
Comme $\Sigma$ est compact et contenu dans~$U$,
la fonction $\log \norm s^{-1}$ est bornée sur~$\Sigma$;
par convergence dominée, on en déduit la convergence de
$ \langle [\log \norm s^{-1} \ddc(u_1^k)  \wedge
\dots \wedge \ddc(u_m^k)],  \alpha\rangle  $ 
vers $ \langle [\log \norm s^{-1} \ddc(u_1)  \wedge
\dots \wedge \ddc(u_m)],  \alpha\rangle  $.
\end{proof}

\begin{prop}\label{prop.masse-totale}
Soit $X$ un $n$-espace analytique sans bord, 
soit $L_1,\dots,L_n$ des fibrés en droites sur~$X$;
pour tout~$i$, soit $\norm\cdot$ et $\norm\cdot '$
des métriques lisses (resp. dpsh) sur~$L_i$.
Il existe une $(n-1,n-1)$-forme lisse~$\alpha$ 
(resp. un $(n-1,n-1)$-courant~$\alpha$) sur~$X$ vérifiant
\[ c_1(\overline L'_1)\wedge\dots\wedge c_1(\overline L'_n)
-  c_1(\overline L_1)\wedge\dots\wedge c_1(\overline L_n)
 = \ddc \alpha.\]
En particulier, si $X$ est propre et de dimension~$n$,
le nombre réel
\[ \int_X c_1(\overline L_1)\wedge\dots\wedge c_1(\overline L_n)  \]
ne dépend que des fibrés en droites~$L_1,\dots,L_n$ sur~$X$,
et pas des métriques dpsh dont ils sont munis. 
\end{prop}
\begin{proof}
Pour tout~$i$, soit $u_i$ la fonction lisse (resp. dpsh)
sur~$X$ telle que $\log\norm{s}=u_i + \log\norm s'$ pour toute
section inversible locale~$s$  de~$L_i$.
On a $c_1(\overline L'_i)=\ddc u_i+c_1(\overline L_i)$.
Par suite,
$c_1(\overline L'_1)\wedge\dots\wedge c_1(\overline L'_n)
-  c_1(\overline L_1)\wedge\dots\wedge c_1(\overline L_n)$
est somme de termes de la forme
$T_1\wedge T_2\wedge \ldots T_n$ 
où chacun chacun des $T_i$ est un courant ou bien égal à 
$c_1(\overline{L_i})$ ou bien à $\ddc u_i$, le second cas se
présentant pour au moins un indice $i$. 
Considérons un tel terme, et soit $j$
tel que $T_j=\ddc u_j$. 
Alors $T_1\wedge T_2\wedge \ldots T_n$ est égal à
$\ddc (u_j\wedge T_1\wedge T_2\wedge T_{j-1}\wedge 
T_{j+1}\wedge\ldots \wedge T_n)$. 
C'est donc une forme lisse, resp. un courant symétrique fermé, 
de bidegré $(n-1,n-1)$. 
Il en résulte que $c_1(\overline L'_1)\wedge\dots\wedge c_1(\overline L'_n)
-  c_1(\overline L_1)\wedge\dots\wedge c_1(\overline L_n)$
est bien de la forme $\ddc \alpha$ où $\alpha$ est une forme lisse, resp. un courant symétrique fermé, 
de bidegré $(n-1,n-1)$. 

Par conséquent 
\[ \int_X c_1(\overline L_1)\wedge\dots\wedge c_1(\overline L_n)
=\int_X c_1(\overline L_1)\wedge\dots\wedge c_1(\overline L_n)+
\int_X\ddc \alpha\]
et l'on conclut en remarquant que comme $X$ est propre
\[\int_X\ddc \alpha=\langle \ddc \alpha, 1\rangle=\langle
\alpha, \ddc 1\rangle=0.\]
Cela conclut la démonstration.
\end{proof}

\begin{theo}\label{theo.MA-lisse-masse}
Soit $X$ une $k$-variété algébrique propre, purement de dimension~$n$.
Soit $\bar L_1,\ldots, \bar  L_n$ une famille de 
fibrés en droites munis de métriques dpsh sur~$X^\an$
(par exemple, des métriques lisses ou formelles).
On a
\[ \int_{X^\an} c_1(\bar L_1)\wedge\dots \wedge c_1(\bar L_n)
=  \int_X c_1(L_1)\cdots c_1(L_n). \]
\end{theo}
Remarquons que les hypothèses entraînent que $X^\an$ 
est un $n$-espace $k$-analytique sans bord.
Précisons également que comme $X$ est propre, chacun des~$L_i$ 
est l'analytifié d'un fibré en droites canonique
sur~$X$ (\cite{poineau2010a}, Appendice A)  que l'on note encore~$L_i$, 
ce qui donne un sens au  membre de droite. 
Cela prouve en particulier que le membre de gauche de la relation
précédente est un entier relatif, positif ou nul si les~$L_i$
sont nef, ce qui n'est pas évident a priori. Nous étendrons
cette propriété d'intégralité au cas des espaces analytiques propres
arbitraires, non nécessairement algébriques
(corollaire~\ref{coro.integralite}).  

\begin{proof}
D'après les propositions~\ref{prop.masse-totale} et~\ref{prop.existence-metrique-lisse},
on peut supposer que les métriques des fibrés~$L_i$ sont lisses.
Soit $(X_j)$ la famille des composantes irréductibles réduites de~$X$, 
et pour tout $j$ soit~$m_j$ la multiplicité générique de~$X$ le long de~$X_j$. 
Les~$X_j^\an$ sont alors les composantes irréductibles de~$X^\an$
et $X^\an$ est de multiplicité générique~$m_j$ le long de~$X_j^\an$
(\ref{ss-composantes-irreductibles} et lemme~\ref{lem-gaga-multiplicite}). 
On a alors
\[\int_{X^\an}c_1(\bar L_1)\cdots c_1(\bar L_n)
=\sum_j m_j \int_{X_j^\an}c_1(\bar L_1)\cdots c_1(\bar L_n)\]
et
\[
\int_Xc_1(L_1)\cdots c_1(\L_n)
=\sum_j m_j \int_{X_j}c_1(L_1)\cdots c_1(L_n),\]
la première de ces égalités résultant de la proposition \ref{prop.formule-degre-integral}, 
et la seconde de la théorie de la théorie de l'intersection standard. 
Il suffit donc de prouver l'égalité voulue pour chacune des $X_j$, ce qui permet
de se ramener au cas où $X$ est intègre. On raisonne alors par récurrence
sur $n$.

Si $n=0$ alors $X$
est le spectre d'une extension finie $L$ de $k$, et les deux termes de l'égalité à démontrer sont alors
égaux à $[L:k]$ : 
cela découle de la proposition~\ref{prop.formule-degre-integral}
du côté analytique, et de la définition du degré
d'intersection du côté algébrique. 
Le théorème est donc vrai lorsque $n=0$. 

Supposons $n>0$ et le résultat vrai en dimension $n-1$. 
Choisissons une section rationnelle non nulle $s$ de~$L_n$.
Soit $D$ le diviseur de $s$ vue comme section rationnelle inversible
de $L_n$ sur le schéma $X$, et soit et soit $D^\an$ le diviseur de $s$
vue comme section méromorphe inversible de
$L_n$ sur l'espace analytique $X^\an$.  Écrivons $D=\sum m_j D_j$ où les $D_j$ 
sont irréductibles et réduits ; par les mêmes arguments que plus haut, on a 
$D^\an=\sum m_j D_j^\an$. 

Alors, $\log\norm s$ est un courant sur~$X^\an$ et l'on a la formule
de Poincaré--Lelong (théorème~\ref{theo.poincare-lelong-metriques})
\[ \ddc [\log\norm s^{-1}]+ \delta_{D^\an}=c_1(\bar L_n). \]
Notons $\alpha$
la forme lisse
$c_1(\bar L_1)\dots c_1(\bar L_{n-1})$, et appliquons-lui cette
égalité de courants. Il vient
\[ \langle [\log\norm s^{-1}],\ddc\alpha\rangle
+ \int_{D^\an} \alpha = \int_X c_1(\bar L_1)\dots c_1(\bar L_n). \]

On a par ailleurs
\begin{align*}
\int_{D^\an}\alpha &=\sum m_i \int_{D_i^\an}\alpha\\
&=\sum m_i \int_{D_i}c_1(L_1)\cdots c_1(L_{n-1}) \\
&=\int_X c_1(L_1)\cdots c_1(L_n),
\end{align*}
la première égalité résultant
de la définition de l'intégrale sur un cycle, la seconde
de l'hypothèse de récurrence, 
et la troisième du fait que $c_1(L_n)\cap [X]=D$. 
D'autre part, la forme $\alpha$ s'écrit localement
\[ \ddc (u_1)\dots \ddc(u_{n-1}), \]
où $u_i=\log\norm{s_i}^{-1}$, $s_1,\dots,s_{n-1}$ étant des sections inversibles de $L_1,\dots,L_{n-1}$.
Il s'ensuit que $\ddc \alpha=0$, ce qui permet de conclure. 
\end{proof}


\section{Opérateur de Monge-Ampère (métriques paralinéaires)}

\subsection{}
Le théorème principal de cette section affirme que l'opérateur de Monge-Ampère
que nous avons défini dans cet article 
coïncide avec celui défini
par le premier auteur dans~\cite{chambert-loir2006}.
La théorie construite dans ce dernier article partait
du cas des métriques formelles, alors considérées comme métriques
{\og lisses\fg} : la théorie de l'intersection sur la fibre spéciale
fournissait des mesures atomiques pour des fibrés en droites
munis de métriques formelles nef ; un argument d'approximation
permettait de construire des mesures  pour des fibrés en droites
munis de métriques limites uniformes de métriques nef.
Dans le présent article, le point de vue est renversé : les
métriques formelles sont singulières mais (lorsqu'elles 
sont nef) sont limites uniformes de métriques psh lisses, 
et leurs mesures de Monge-Ampère
se calculent par approximation.

\subsection{}
Soit $X$ un espace $k$-analytique purement de dimension $n$. 
Rappelons que nous désignons
par $X^\#$ l'ensemble des points $x$ de $X$ tels que $d_k(x)=n$ ; nous noterons
$X_{\mathrm{str}}^\#$ le sous-ensemble de $X^\#$ formé des points tels que 
$\hrt x^1$ soit de degré de transcendance $n$ sur $\widetilde k^1$. 
Si $x\in X^\#$ alors $\hrt x$ est de type fini sur $\hrt k$ (lemme
\ref{lemm.abhyankar-finitude}); si 
$x\in X^\#_{\mathrm{str}}$ alors $\hrt x$ est fini sur $\hrt x^1
\otimes_{\widetilde k^1}\widetilde k$, 
de degré $(\abs{\hrt x}^\times : \abs {k^\times})$. 

\begin{prop}\label{prop.MA-berk}
Soit $X$ un $n$-espace $k$-analytique sans bord.
Soit $(L_1,\ldots, L_n)$ une famille de 
fibrés en droites métrisés paralinéaires entiers
sur~$X$ 
et soit~$S$ le sous-ensemble de~$X$ 
formé des points $x$ tels que la dimension de définition 
simultanée des $\widetilde {L_{i,x}}$ 
soit égale à~$n$.  Il est
fermé et discret dans~$X$, et contenu dans~$X^\#$.
Les fibrés métrisés~$L_i$ sont dpsh
et
il existe une famille $(\lambda_x)_{x\in S}$ 
de nombres réels telle que 
\[\bigwedge_{i=1}^n c_1(L_i)=\sum_{x\in S} \lambda_x \delta_x.\] 
\end{prop}
Dans la suite de ce paragraphe, nous donnerons une formule
pour les nombre réels~$\lambda_x$ en fonction
des fibrés résiduels~$\widetilde {L_{i,x}}$. 
Ces formules montreront notamment 
que ce sont des nombres entiers.

\begin{proof}
La proposition~\ref{prop.dimdefn} garantit
que l'ensemble des points $x$ tels que la dimension de définition 
simultanée des $\widetilde{L_{i,x}}$ soit égale à $n$ est fermé et discret dans $X$, et sa
définition même assure qu'il est inclus dans
$X^\#$. 
De plus, les fibrés métrisés~$L_i$ sont tropicalement d-convexes
d'après la proposition~\ref{prop.metrique-pl-approchable}, et en particulier dpsh. 

Soit $x\in X$. Démontrons qu'au voisinage de~$x$,
le courant indiqué est de la forme~$\lambda_x \delta_x$.

Choisissons une famille $(V,M_1,\dots,M_n)$ où,
pour pour $i$, $(V,M_i)$ est un prémodèle de $\widetilde {L_{i,x}}$,
et $V$ est de dimension minimale, 
c'est-à-dire la dimension de définition simultanée des~$\widetilde L_{i,x}$.

Par multilinéarité, 
on se ramène au cas où les $M_i$ sont tous engendrés par leurs sections globales
(grâce à \ref{ss.quasiproj-difference}) puis, 
par polarisation, au cas où 
les $M_i$ sont tous égaux à un même fibré en droites~$M$.
Alors, au voisinage de~$x$, les fibrés métrisés~$L_i$
sont égaux à un même fibré métrisé~$L$.

Soit $s$ une section inversible de $L$ au
voisinage de $x$ . 
D'après le corollaire~\ref{coro-engendre-sections-globales},
$\norm{s}^{-1}$ est au voisinage de $x$ de la forme
$ \max( r_1\abs{a_1},\dots,r_m\abs{a_m})$
où les $r_i$ sont des réels strictement positifs
et les $a_i$ des éléments de~$\mathscr O_{X,x}^\times$
tels que $\widetilde{a_i/a_j}\in \tilde k(V)$ pour tout couple~$(i,j)$. 
D'après la prop.~\ref{prop.ma-berk-local},
le courant $c_1(L)^n$ est donc de la forme~$\lambda_x\delta_x$
au voisinage de~$x$.

Il reste à vérifier que $\lambda_x=0$ si $x\not\in S$.
Dans ce cas, le choix de~$V$
entraîne que sa dimension est strictement inférieure à~$n$.
D'après le théorème~3.4 de~\cite{ducros2012b},
la dimension tropicale en~$x$ de la famille $(a_i/a_1)_i$ est strictement
inférieure à~$n$.
D'après la proposition~\ref{prop.ma-berk-local},
la mesure~$c_1(L)^n$ est nulle au voisinage de~$x$.
\end{proof}

\begin{rema}\label{rem-cl-strict}
Supposons que les fibrés métrisés paralinéaires entiers $L_i$ sont stricts. 
Soit $x\in S$. Les fibrés $\widetilde {L_{i,x}}$ sont stricts. 
Si $F$ désigne un sous-corps gradué de $\hrt x_1$
contenant $\widetilde k^1$ et
sur lequel tous les $\widetilde{L_{i,x}}^1$ sont définis, alors tous les $\widetilde{L_{i,x}}$ sont
définis sur $F\otimes_{\widetilde k^1} \widetilde k$. Ce dernier est donc de degré de transcendance $n$
sur $\widetilde k$, ce qui force $F$ à être de degré de transcendance $n$ sur $\widetilde k^1$. 
Ainsi  l'ensemble $S$ est contenu dans l'ensemble des points $x$ tels que la dimension de définition 
simultanée des $\widetilde{L_{i,x}}^1$ soit égale à $n$.
\end{rema}

\begin{prop}\label{prop.ma-berk-alg-clos}
Supposons $k$ algébriquement clos.
Soit $X$ un $n$-espace $k$-analytique sans bord et génériquement réduit, 
et soit $(L_1,\ldots, L_n)$ une famille de 
fibrés en droites métrisés paralinéaires
strictement entiers sur~$X$.
Le support de la mesure discrète $\bigwedge c_1(L_j)$
est contenu dans $X^\#_{\mathrm{str}}$. 
La masse d'un point~$x$ de $X^\#_{\mathrm{str}}$
pour cette mesure est égale à
\[
\int_{\zr {\hrt x^1}{\widetilde k^1}}
c_1(\widetilde {L_{1,x}}^1)\cdots c_1(\widetilde {L_{n,x}}^1)\]
\end{prop}
Rappelons que sous nos hypothèses les fibrés en
droites résiduels $\widetilde{L_{i,x}}$ sont stricts 
que $\widetilde{L_{i,x}}^1$ désigne le fibré en droites sur $\zr {\hrt x^1}{\widetilde k^1}$
dont provient $\widetilde{L_{i,x}}$.
Pour une  description directe de ces fibrés en droites, 
voir le \S\ref{ss.description-fibre-residuel-strict}.

\begin{proof}
On sait d'après la proposition~\ref{prop.MA-berk},
que le support de cette mesure est fermé et discret. 
On sait aussi que si $x$ est un point de ce support, 
la dimension de définition résiduelle simultanée
des $\widetilde{L_{i,x}}$ est égale à $n$. Comme les $\widetilde{L_{i,x}}$ sont
par ailleurs stricts, ils sont tous définis sur $\hrt x^1\cdot \widetilde k$, ce qui
implique que $\hrt x^1$ est de degré de transcendance $n$ sur $\widetilde k^1$. 

\textit{Comme nous n'allons manipuler dans la preuve que des réductions strictes
et des objets homogènes de degré $1$
nous nous permettons exceptionnellement 
dans ce qui suit de ne plus noter le $1$ en exposant : nous écrirons
ainsi $\widetilde k, \hrt x, \widetilde{L_{i,x}}$, etc. au lieu de 
$\widetilde k^1, \hrt x^1, \widetilde{L_{i,x}}^1$, etc.}

Soit $x\in X^\#_{\mathrm{str}}$. 
Avec les notations de la proposition~\ref{prop.MA-berk},
il s'agit de démontrer
l'égalité $\lambda_x = 
\int_{\zr{\hrt x}{\widetilde k}} c_1(\widetilde {L_{1,x}})\cdots c_1(\widetilde {L_{n,x}})$. 
Comme dans le début de la preuve
de la proposition~\ref{prop.MA-berk}, on se ramène au cas où
tous les fibrés métrisés~$L_i$
sont égaux à un même fibré métrisé~$L$
(paralinéaire strictement entier)
dont le fibré résiduel~$\widetilde{L_x}$ est engendré par ses sections
globales. 
La métrique de $L$ étant paralinéaire strictement 
entière, $L$ possède une section $s$ inversible au voisinage
de $x$ telle que $\norm{s(x)}=1$. 
D'après le corollaire~\ref{coro-engendre-sections-globales-str}, 
il existe une famille finie $(f_1,\dots,f_m)$
d'éléments de $\mathscr O_{X,x}^\times$ 
telle que $\abs{f_i(x)}=1$ pour tout $i$ et telle que 
\begin{equation}
 \log (\norm{s}^{-1}) = \max(\abs{f_1},\dots,\abs{f_m})
\end{equation}
au voisinage de~$x$.
Comme le calcul est local, il n'est pas restrictif de supposer que les~$f_i$
sont définies sur~$X$; notons $f\colon X\to\gm^m$ le moment~$(f_1,\dots,f_m)$.
D'après la proposition~\ref{prop.ma-berk-local},
$\lambda_x$ est égal à l'invariant~$\lambda$ du cône tropical~$\Pi_{f,x}$
du moment~$f$ en~$x$.

Par choix de $x$ on a $d_k(x)=n$, si bien que 
l'anneau local $\mathscr O_{X,x}$ est artinien; l'espace~$X$
étant génériquement réduit, l'anneau~$\mathscr O_{X,x}$ est réduit
et est donc un corps ; en particulier,
l'anneau $\mathscr O_{X,x}$ est régulier. 
L'espace~$X$ est dès lors géométriquement régulier en~$x$, donc
lisse au voisinage de~$x$ (puisqu'il est sans bord). 
Quitte à remplacer~$X$ par un voisinage ouvert de~$x$, 
on peut alors supposer qu'il
existe un $k$-schéma de type fini affine intègre lisse~$\mathscr X$
et un ouvert de~$\mathscr X^\an$ qui est isomorphe à~$X$
(voir par exemple le lemme~(0.21) de~\cite{ducros2012b}). 
Comme $k(\mathscr X)$ est dense dans $\hr x$,
on peut également supposer que les $f_i$ appartiennent
à l'anneau $\mathscr O(\mathscr X)^\times$,
et donc que $X=\mathscr X^\an$.
Alors, le moment~$f$ considéré ci-dessus provient
d'un moment algébrique $f\colon \mathscr X\to \gm^m$.


Le
lemme \ref{lemm-secglob-max} assure alors que les $\widetilde{f_i(x)s(x)}$ 
sont des sections globales de 
$\widetilde{L_x}$ qui engendrent ce dernier. 
Soit $(M,V)$ 
un modèle propre de $(\widetilde{(X,x)}, \widetilde{L_x})$. 
On a par définition 
\[ \int_{\zr {\hrt x}{\widetilde k}}c_1(\widetilde{L_x})^n
 =\int_Vc_1(M)^n. \]

Munissons le corps résiduel $\widetilde k$ de la valeur absolue triviale;
cela permet de considérer l'analytifiée~$V^\an$ 
de la $\widetilde k$-variété propre~$V$.
Notons $\xi$ le point de~$V^\an$ 
qui correspond à la valeur absolue triviale sur $\widetilde k(V)$;
on a $\hrt\xi=\hrt x$. 
Pour tout~$i$, on considère $\widetilde{f_i(x)}$ 
comme une fonction analytique sur le germe~$(V^\an,\xi)$;
soit $\widetilde f\colon (V^\an,\xi)\to \mathbf G^m_{\mathrm m,\tilde k}$
le moment défini par ces fonctions.
On considère aussi $\widetilde s(x)$ 
comme une section du fibré analytique~$M^\an$
sur le germe $(V^\an,\xi)$.  

En considérant $V$ comme un schéma formel sur l'anneau topologique
discret $\widetilde k$, on peut voir $M^\an$ comme la fibre générique $M_\eta$
 du fibré formel $M$, fibre générique qui hérite d'une métrique paralinéaire
strictement entière
(exemple~\ref{fibvaltriv}) ; 
et le fibré résiduel correspondant 
$\widetilde {M_{\eta,\xi}}$ sur $\zr{\hrt \xi}{\widetilde k}=\zr {\hrt x}{\widetilde k}$
possède alors comme prémodèle $M$ lui-même (lemme \ref{lemme-premodele-formel}) ; 
il résulte des constructions que pour toute section rationnelle $u$ non nulle de $M$
sur un ouvert de Zariski de $V$ on a $\norm{u(\xi)}=1$
et $\widetilde {u(\xi)}=u$, lorsqu'on identifie 
$\widetilde{M_{\eta,\xi}}$ au fibré en droites de modèle $M$. 
Puisque les $\widetilde{s(x)f_i(x)}$ sont des sections globales de $\widetilde{L_x}
=\widetilde{M_{\eta,\xi}}$ qui engendrent ce dernier, on déduit alors du corollaire
\ref{coro-max-ai-str} que 
\begin{equation}
\log(\norm{\widetilde s(x)}^{-1}) = \max (\abs{\widetilde{f_1(x)}},\dots,\abs{\widetilde{f_m(x)}})
\end{equation}
au voisinage de $\xi$ sur $V^\an$. 

D'après le théorème~\ref{theo.comparaison-cones},
les cônes tropicaux $\Pi_{f,x}$ et $\Pi_{\tilde f,\xi}$
sont égaux.
D'après la proposition~\ref{prop.ma-berk-local},
on a donc $\lambda_x=\lambda_\xi$, où
$\lambda_\xi$ est le coefficient de $\delta_\xi$
apparaissant dans la mesure discrète $c_1(M_\eta)^n$. 

Observons maintenant
que le point $\xi$ est le seul point de $V^\an$ qui ait un corps 
résiduel de degré de transcendance égal à $n$ sur $\widetilde k$. 
En effet, si $\omega$ est un tel point, on a $d_{\tilde k}(\omega)=n$;
comme $\dim(V)=n$, le point~$\omega$ est situé au-dessus du point générique 
de~$V^\an$ et $\abs{\hr{\omega}^\times}$ est de torsion modulo
$\abs{\tilde k^\times}$; ainsi, $\omega$ correspond à la valeur
absolue triviale sur~$\tilde k(V)$, donc $\omega=\xi$.
On a en conséquence l'égalité
\[ c_1(\overline {M_\eta})^n = \lambda_\xi \delta_\xi \]
de mesures sur~$V^\an$, où $\lambda_\xi$ est l'invariant~$\lambda$
du cône tropical~$\Pi_{\tilde f,\xi}$.
D'autre part, la masse totale de~$c_1(\overline {M_\eta})^n$ est égale
au degré $\int_V c_1(M)^n$
en vertu du théorème~\ref{theo.MA-lisse-masse}
(et du fait que $M_\eta=M^{\an}$ comme fibré en droites). 
Cela conclut la démonstration de la proposition.
\end{proof}

\begin{rema}\label{rem-monge-ampere-algclos}
Conservons les hypothèses de la proposition 
\ref{prop.ma-berk-alg-clos}. Soit $x\in X^\#_{\mathrm{str}}$. 
Le groupe $\abs{\hrt x}^\times/\abs{k^\times}$ est alors de torsion, et est
dès lors trivial puisque $\abs{k^\times}$ est divisible. On a donc
$\hrt x=\hrt x^1\otimes_{\widetilde k^1}\widetilde k$, si bien que 
\[\int_{\zr {\hrt x^1}{\widetilde k^1}}
c_1(\widetilde {L_{1,x}}^1)\cdots c_1(\widetilde {L_{n,x}}^1)
=\int_{\zr {\hrt x}{\widetilde k}}
c_1(\widetilde {L_{1,x}})\cdots c_1(\widetilde {L_{n,x}}).\]
Nous avons préféré donner la description non graduée de cet entier dans l'énoncé
de la proposition car sa définition est plus élémentaire, 
et car c'est sous cette forme qu'elle
apparaît au cours de la démonstration. 

Mais nous nous proposons maintenant de décrire la mesure
$c_1(L_1)\wedge \ldots \wedge c_1(L_n)$ 
(autrement dit, de calculer les coefficients $\lambda_x$ de la 
proposition \ref{prop.MA-berk})
dans le cas général, c'est-à-dire sans supposer
que $k$ est algébriquement clos, ni que $X$ est génériquement réduit, ni que les métriques 
paralinéaires entières en jeu sont strictes. L'expression naturelle des $\lambda_x$ se fera alors
au moyen des nombres d'intersections gradués, et sera établie en se ramenant 
par extension des scalaires au cas particulier 
que nos venons de traiter dans la proposition~\ref{prop.ma-berk-alg-clos}; 
il sera en conséquence utile d'avoir remarqué que les nombres d'intersections 
non gradués de cette proposition
peuvent tout aussi bien s'interpréter comme des nombres d'intersection gradués. 
\end{rema}

%

\begin{theo}\label{theo.formule-ma-berk}
Soit $X$ un $n$-espace $k$-analytique sans bord et
soit $(L_1,\ldots, L_n)$ une famille de 
fibrés en droites métrisés
paralinéaires entiers sur~$X$.
Le support de la mesure discrète $\bigwedge c_1(L_j)$
est contenu dans $X^\#$. 
La masse de tout $x\in X^\#$ pour cette mesure est égale à l'entier
\[
m(x)\cdot\df_k(x)\cdot
\int_{\zr{\hrt x}{\widetilde k}}c_1(\widetilde {L_{1,x}})\cdots
c_1(\widetilde{L_{n,x}}),\]
où $m(x)$ désigne la multiplicité de l'unique composante irréductible de $X$ qui contient $x$. 
\end{theo}
\begin{proof}
D'après la proposition~\ref{prop.MA-berk}, le support de
cette mesure est contenu dans l'ensemble discret~$S$ 
des points~$x\in X$
tels que la dimension de définition simultanée des~$\widetilde{L_{i,x}}$ 
soit égale à~$n$, et a fortiori dans $X^\#$. 
Il suffit donc de vérifier l'assertion au voisinage de tout point $x$ de $X^\#$.

Soit $x$ un point de $X^\#$ ; 
notons $\lambda_x$ sa masse du point~$x$ pour 
la mesure $\bigwedge c_1(L_j)$. Comme $x$ est situé sur $X^\#$, il appartient à une unique composante irréductible $X'$
de $X$. Soit $X''$ l'ouvert de Zariski de $X$ constitué des points de $X'$ qui n'appartiennent à aucune autre composante. 
C'est un espace irréductible (\cf \cite{ducros2021}, Lemme 1.11) ; pour calculer $\lambda_x$ on peut remplacer $X$
par $X''$, et donc supposer $X$ irréductible. 
Choisissons une famille finie de domaines affinoïdes
$(V_j)$ de $X$, contenant $x$ et dont la réunion est un voisinage de $x$, 
et tel que pour tout $i$ le fibré $L_i|_{V_j}$ possède sur 
$V_j$ une section dont la norme appartient à $\rho_{ij}\cdot \abs{\mathscr O_X(V_j)^\times}$
pour un certain réel strictement positif
$\rho_{ij}$.
Soit $r=(r_\ell)$ un polyrayon de longeur finie contenant tous les $\rho_{ij}$ ainsi
qu'un système générateur de $(\abs{\hr x}^\times/\abs{k^\times})\otimes \Q$, et
soit $k_r$
l'extension complète $\hr{\eta_r}$ de $k$. 
Soit $S$ le support de $\bigwedge c_1(L_j)$
et soit $\sigma$
la section de Shilov de l'application naturelle $p\colon X_r\to X$ (où
$X_r=X\times_k k_r$).
Posons
$M_i=p^*L_i$. 
La proposition \ref{rema.courants-chgt-base}, couplée à la déifnition 
des courants $c_1(L_i)$, assure que le support de 
$\bigwedge c_1(M_i)$ (vue comme mesure sur l'espace
$k_r$-analytique
$X_r$) coïncide avec $\sigma(S)$. Au vu de la proposition
\ref{prop.bt-berk-projection}, ceci entraîne que $\lambda_x$ est égal 
au poids de la mesure $\bigwedge c_1(M_i)$  en $\sigma(x)$. 
Et puisque $\hrt{\sigma(x)}$ est simplement égal
au corps (gradué) des fractions de 
$\hrt x\otimes_{\widetilde k}\widetilde{k_r}$, 
on a 
\[\int_{\zr{\hrt x}{\widetilde k}}c_1(\widetilde {L_{1,x}})\cdots
c_1(\widetilde {L_{n,x}})
=\int_{\zr{\hrt {\sigma(x)}}{\widetilde {k_r}}}c_1(\widetilde {M_{1,\sigma(x)}})\cdots
c_1(\widetilde {M_{n,\sigma(x)}}).\]
Il suffit donc
de montrer que la masse de $\bigwedge
c_1(M_i)$ en $\sigma(x)$
est égale 
à 
\[m(x)\cdot\df_k(x)\cdot \int_{\zr{\hrt{\sigma(x)}}{\widetilde{k_r}}}
c_1(\widetilde {M_{1,\sigma(x)}})\cdots
c_1(\widetilde {M_{n,\sigma(x)}}).\]
Par ailleurs $m(x)=m(\sigma(x))$
car l'extension $k_r$ de $k$ est analytiquement séparable, 
donc ne fait pas apparaître de nouvelles multiplicités
(corollaire \ref{coro-longueur-artinien}). 
Et $\df_{k_r}(\sigma(x))=\df_k(x)$ 
d'après \ref{ss-blabla-defaut}. 
Il suffit donc de démontrer la formule
de l'énoncé pour les fibrés métrisés
$M_i$ sur l'espace $k_r$-analytique $X_r$
en le point $\sigma(x)$. Autrement dit, 
on s'est ramené au cas où les $L_i$ sont
paralinéaires strictement entiers au voisinage de $x$
et où $x\in X^\#_{\mathrm{str}}$. 

Soit $K$ le complété d'une clôture algébrique de~$k$.
Soit $j\colon X_{K,\red}\to X_K$ l'immersion
fermée canonique
et $q\colon X_K \to X$ le morphisme de changement de base; 
posons $p=q\circ j$.

Notons~$m$ la multiplicité générique de~$X$; c'est l'entier~$m(x)$
du théorème.
Le groupe~$\Aut(K/k)$ permute transitivement
les composantes irréductibles de~$X_K$, si bien qu'elles
ont toute même multiplicité générique~$m_K$.
Par suite, le degré calibré de~$j$ est identiquement égal à~$1/m_K$.
Appliquant la proposition~\ref{prop.image-directe-fini} 
à l'immersion fermée de~$X_\red$ dans~$X$ et à $d=1/m_K$, 
on en déduit
\[ \bigwedge c_1(q^*L_j) = m_{K} j_* (\bigwedge c_1(p^*L_j)). \]
De plus, par changement de base, on a  l'égalité de mesures
\[ \bigwedge c_1(L_j) = q_* (\bigwedge c_1(q^* L_j)). \]
Ainsi,  on a
\[ \bigwedge c_1(L_j) = m_{K} p_* (\bigwedge c_1(p^* L_j)). \]

Appliquons la proposition~\ref{prop.ma-berk-alg-clos} aux
fibrés $p^*L_j$ sur~$X_{K,\red}$; on obtient
\[ \bigwedge c_1(p^*L_j) =
\sum_{y\in X_{K,\red,\mathrm{str}}^\#}
  \left( \int_{\zr{\hrt y^1}{\widetilde K^1}}
 c_1(\widetilde {M_{1,y}}^1)\cdots c_1(\widetilde{M_{n,y}}^1)\right) \delta_y. \]
Compte de la remarque~\ref{rem-monge-ampere-algclos},
on a donc
\[ \bigwedge c_1(p^*L_j) =
\sum_{y\in X_{K,\red,\mathrm{str}}^\#}
  \left( \int_{\zr{\hrt y}{\widetilde K}}
 c_1(\widetilde {M_{1,y}})\cdots c_1(\widetilde{M_{n,y}})\right) \delta_y. \]

Puisque les antécédents de~$x$ par~$p$ appartienennt à tous
à~$X^\#_{K,\red,\mathrm{str}}$,
il vient
\[ \lambda_x = m_K  \sum_{y \in p^{-1}(x)} 
  \int_{\zr{\hrt y}{\widetilde K}}
 c_1(\widetilde {M_{1,y}})\cdots c_1(\widetilde{M_{n,y}}).\]

Notons $\widetilde \ell$ la longueur de l'algèbre artinienne
$\hrt x\otimes_{\widetilde k} \widetilde K$
et $\nu$ le cardinal de $q^{-1}(x)$. 
Fixons $y\in p^{-1}(x)$.
Le corps gradué~$\hrt y$ est fini sur 
$\hrt x\otimes \widetilde K$; son degré 
sur l'extension composée $\hrt x\cdot \widetilde K$ dans~$\hrt y$
ne dépend que de~$x$; on le note~$\widetilde d$. 
Dans la suite, nous écrirons $\widetilde{L_i}$
au lieu de $\widetilde {L_{i,x}}$ et
$\widetilde {M_i}$ au lieu de $\widetilde {M_{i,y}}$, et nous noterons
$N_i$ l'image réciproque de $\widetilde{L_i}$ sur $\zr {\hrt x\cdot \widetilde K}
{\widetilde K}$. 

Soit $\nu'$ le nombre d'idéaux
maximaux de  
$\hrt x\otimes_{\widetilde k}\widetilde K$. Cette dernière algèbre
est alors produit de 
$\nu'$~algèbres artiniennes locales, 
ayant toutes même longueur $\widetilde \ell/\nu'$, et 
toutes même corps gradué résiduel isomorphe
à $\hrt x\cdot \widetilde K$. 

Si $\mathsf X$ est un modèle de~${\zr{\hrt x}{\widetilde k}}$,
alors $\mathsf X_{\tilde K}$ possède~$\nu'$ composantes
irréductibles, toutes de multiplicité générique~$\widetilde\ell/\nu'$
et toutes de même corps gradué de fonctions~${\hrt x\cdot \widetilde K}$.
En appliquant successivement
les assertions~\emph b) et~\emph c)
de la proposition~\ref{proprietes-intersection-graduee},
on obtient l'égalité
\begin{align*}
 \int_{\zr{\hrt x}{\widetilde k}}
c_1(\widetilde{L_{1,x}})\cdots c_1(\widetilde{L_{n,x}})
& = \nu' \cdot \frac{\widetilde\ell}{\nu'}
\int_{\zr{\hrt x\cdot \widetilde K}
{\widetilde K}}c_1(N_1)\cdots c_1(N_n) \\
& =
\widetilde\ell
\int_{\zr{\hrt x\cdot \widetilde K}
{\widetilde K}}c_1(N_1)\cdots c_1(N_n).
\end{align*}
En appliquant l'assertion~\emph d) de cette proposition, on en déduit
\[
\int_{\zr{\hrt x}{\widetilde k}}
c_1(\widetilde{L_{1,x}})\cdots c_1(\widetilde{L_{n,x}})
=\frac{\widetilde \ell}{\widetilde d}\int_{\zr{\hrt y}{\widetilde K}}
c_1(\widetilde {M_1})\cdots c_1(\widetilde{M_n}).
\]


Par suite, on a
\[ \lambda_x = 
 \frac{\nu\cdot m_K\cdot \widetilde d}{\widetilde \ell}
 \int_{\zr{\hrt x}{\widetilde k}}
c_1(\widetilde{L_{1,x}})\cdots c_1(\widetilde{L_{n,x}}). \]

Notons $m'_K$ la multiplicité générique commune 
des composantes irréductibles de~$X_{\red, K}$.
D'après la proposition \ref{prop.calcul-defaut} \emph a), \emph c) et \emph e), l'anneau est artinien 
$\hr x\widehat \otimes_k K$
et sa longueur $\ell$ est égale à $\nu m'_K$ ; et d'après la proposition
 \ref{prop.calcul-defaut}  \emph b)  et la proposition \ref{prop.abh-default-gal}, le défaut $\df_k(x)$ est égal 
à $\ell \widetilde d/\widetilde \ell$. 

Il vient 
\begin{align*}
 \lambda_x 
&=\frac{m_K}{m'_K}\cdot \frac{\nu \cdot m'_K\cdot \widetilde d}{\widetilde \ell}
\int_{\zr{\hrt x}{\widetilde k}}c_1(\widetilde{L_{1,x}})\cdots c_1(\widetilde{L_{n,x}})\\
&=\frac{m_K}{m'_K}\cdot \frac{\ell \widetilde d}{\widetilde \ell}
\int_{\zr{\hrt x}{\widetilde k}}c_1(\widetilde{L_{1,x}})\cdots c_1(\widetilde{L_{n,x}})\\
&=\frac{m_K}{m'_K}\cdot \df_k(x) \int_{\zr{\hrt x}{\widetilde k}}c_1(\widetilde{L_{1,x}})\cdots c_1(\widetilde{L_{n,x}})\\
&=m\cdot \df_k(x) \int_{\zr{\hrt x}{\widetilde k}}c_1(\widetilde{L_{1,x}})\cdots c_1(\widetilde{L_{n,x}}),
\end{align*}
car $m_K=m\cdot m'_K$ en vertu du lemme \ref{lemme-longueur-XK}.
\end{proof}

%

\begin{exem}
Supposons que le corps~$k$ soit stable. 
Dans ce cas pour toute base d'Abhyankar~$t$ de~$\hr x$ 
le corps~$\widehat{k(t)}$ est stable, 
et $\hr x$ est donc sans défaut sur icelui. 
Par conséquent $\df_k(x)=1$ et l'on a donc
$\lambda_x=m(x)\int_{\zr {\hrt x}{\widetilde k}}
c_1(\widetilde{L_{1,x}})\cdots c_1(\widetilde{L_{n,x}}).$
\end{exem}

\begin{exem}
Supposons toujours que $k$ est stable. 
Faisons de plus l'hypothèse que
le corps résiduel classique
$\hrt x^1$ est de degré de transcendance 
$n$ sur $\tilde k^1$
et que les métriques des $L_i$ sont
paralinéaires strictement entières au voisinage de $x$.
L'indice
$(\abs {\hr x^\times}:\abs{k^\times})$ 
est alors fini, notons-le $e$. 
Pour tout $i$, le
fibré résiduel 
$\widetilde{L_{i,x}}$ est 
le tiré en arrière le long
de l'extension $\hrt x^1\otimes_{\widetilde k^1}
\widetilde k\hookrightarrow
\hrt x$
du fibré résiduel classique $(\widetilde{L_{i,x}})^1$. 
Or, $\hrt x$ est de degré~$e$ sur $\hrt x^1\otimes_{\widetilde k^1}\widetilde k$.
Il vient
\begin{align*}
\lambda_x&=m(x)\int_{\zr {\hrt x}{\widetilde k}}
c_1(\widetilde{L_{1,x}})\cdots c_1(\widetilde{L_{n,x}})
\\
&=em(x)\int_{\zr {\hrt x_1}{\widetilde k_1}}
c_1(\widetilde{L_{1,x}}^1)\cdots c_1(\widetilde{L_{n,x}}^1).
\end{align*}
\end{exem}

\begin{coro}\label{coro.integralite}
Soit $X$ un $k$-espace analytique propre purement de dimension~$n$,
et soit $\overline L_1,\dots,\overline L_n$ des fibrés en droites sur~$X$
munis de métriques dpsh.
Alors
\[ \int_{X^\an} c_1(\overline L_1)\wedge \dots \wedge c_1(\overline L_n)  \]
ne  dépend pas des métriques
psh dont on a muni les~$L_i$ ; 
c'est un entier relatif, qui est positif ou nul 
si les métriques considérées sont psh.
\end{coro}

\begin{proof}
Que le nombre réel obtenu soit indépendant des métriques dpsh considérées
a déjà été vu (proposition \ref{prop.masse-totale}). 
Le calcul $\int_{X^\an}c_1(\overline L_1)\wedge\dots
\wedge c_1(\overline L_n)$ peut alors être fait après extension
des scalaires à une extension complète~$K$ de~$k$
(proposition~\ref{prop.bt-berk-projection}).
Cela permet de supposer que la valuation de~$k$ est non triviale.

La proposition~\ref{prop.exist.metrique.formelle} garantit alors l'existence 
de métriques paralinéaires entières (et même strictement entières)
sur les fibrés en droites~$L_{i}$.
Il suffit donc de démontrer l'intégralité annoncée
lorsque les métriques des~$\overline{L_i}$ sont paralinéaires entières. 

Dans ce cas, le courant
\[ c_1(\overline L_1)\wedge \dots\wedge c_1(\overline L_n) \]
est une somme finie $\sum \lambda_x \delta_x$
(prop.~\ref{prop.MA-berk}) et les $\lambda_x$ sont entiers 
par le théorème  \ref{theo.formule-ma-berk}. 
Par conséquent
\[ \int_{X^\an}  c_1(\overline L_1)\wedge \dots\wedge c_1(\overline L_n) 
= \sum \lambda_x\]
est bien un entier relatif. 

Si les métriques considérées sont psh,
le courant $c_1(\overline L_1)\wedge\dots\wedge c_1(\overline L_n)$
est positif, de sorte que cet entier est positif ou nul.
\end{proof}

\subsection{}
Compte tenu de l'existence de métriques lisses 
(proposition~\ref{prop.existence-metrique-lisse}),
le corollaire ci-dessus fournit 
un accouplement multilinéaire et symétrique
$ \langle \cdot, \cdots,\cdot \rangle \colon \Pic(X)^n \to \Z$. 
Il est invariant par extension des scalaires.

Lorsque $X$ est algébrique,
il coïncide avec le degré d'intersection usuel
(théorème~\ref{theo.MA-lisse-masse}).

\begin{prop}\label{prop-conservation-nombre}
Soit $d$ un entier et soit $p\colon Y\to X$ un morphisme propre et plat, purement de dimension relative $d$, entre espaces
$k$-analytiques. Soient $L_1,\ldots, L_d$ des fibrés en droites sur $Y$. L'application 
\[X\to \Z, \qquad x\mapsto \langle L_1|_{Y_x},\ldots, L_d|_{Y_x}\rangle\] est localement constante. 
\end{prop}
\begin{proof}
L'assertion étant G-locale sur $X$, on peut le supposer bon. Dans ce cas $Y$ est également bon puisque
$p$ est propre et en particulier
sans bord. Il existe alors sur chacun des $L_i$ une métrique lisse, donnant lieu à un fibré
métrisé $\bar L_i$. Soit $\omega$ la $(d,d)$-forme lisse
$c_1(\bar L_1)\wedge\ldots\wedge c_1(\bar L_d)$ (sur l'espace $Y$). On a pour tout $x\in X$ l'égalité
\[ \langle L_1|_{Y_x},\ldots, L_d|_{Y_x}\rangle=\int_{Y_x} \omega.\] 
Il s'ensuit en vertu du théorème \ref{theo-variation-integrales} que $x\mapsto \langle L_1|_{Y_x},\ldots, L_d|_{Y_x}\rangle$
est continue ; puisqu'elle est à valeurs entières, cette fonction est localement constante. 
\end{proof}

\section{Calculs formels}

Dans ce paragraphe, on explique comment calculer des produits
de $c_1(\overline L)$ lorsque la métrique sur~$L$ est donnée par
un modèle formel. Commençons par un cas élémentaire.

\begin{prop}
Soit $X$ un $k$-schéma propre de dimension~$n$
et soit $L$ un fibré en droites sur~$X$.
Considérons un modèle formel propre~$\mathfrak X$ sur~$k^\circ$ de~$X$
et un fibré en droites~$\mathfrak L$ sur~$\mathfrak X$ prolongeant~$L$.
Supposons que $\mathfrak X$ est relativement normal dans
sa fibre générique et que $\mathfrak X_s$ est irréductible;
soit $\eta$ l'unique point de~$X^\an$ dont la réduction est égale
au point générique de~$\mathfrak X_s$.
Munissons~$L$ de la métrique 
paralinéaire strictement entière induite par le modèle~$\mathfrak L$.
On a
\[ c_1(\overline{L})^n  = \left(\int_X c_1(L)^n\right) \delta_\eta. \]
\end{prop}
\begin{proof}
La spécialisation de tout point~$x$ de~$X^\an$ autre que~$\eta$
a une adhérence de dimension $\leq n-1$, donc la dimension de définition 
de~$\widetilde{L_x}$ est au plus~$n-1$ (lemme \ref{lemme-premodele-formel}). 
D'après la proposition \ref{prop.MA-berk} il existe un nombre réel~$\lambda$
tel que 
\[ c_1 (\overline L)^n = \lambda \delta_{\eta}. \]
D'après le théorème~\ref{theo.MA-lisse-masse}, la masse totale
de cette mesure est le degré $\int_X c_1(L)^n$, 
si bien que $\lambda=\int_X c_1(L)^n$.
\end{proof}

\begin{exem}\label{exem.weil-pn}
Prenons le cas du fibré~$\mathscr O(1)$ sur l'espace projectif~$\P^n$,
muni de la métrique de Weil associée au modèle formel 
évident sur~$k^\circ$. Soit $\eta_1$ l'unique point de~$\P^n$
qui se spécialise en le point générique de la fibre spéciale.
D'après la proposition précédente, on a
\[ c_1 (\overline{\mathscr O(1)})^n = \delta_{\eta_1}. \]
\end{exem}

\subsection{}
Soit $\mathfrak X$ un $k^\circ$-schéma formel plat,
localement topologiquement de type fini et purement de dimension relative~$n$.
Notons $X$ sa fibre générique; supposons-la sans bord (cela revient à demander que les composantes
irréductibles de $\mathfrak X_s$ soient propres).

Soit $(\mathfrak L_1,\dots,\mathfrak L_n)$ des fibrés en droites 
sur~$\mathfrak X$; notons $\overline L_1,\dots,\overline L_n$
les fibrés en droites correspondants sur~$X$ munis de leurs métriques
paralinéaires strictement entières. 

On désire calculer la mesure $c_1(\overline L_1)\wedge\dots\wedge
c_1(\overline L_n)$ sur~$X$ de manière analogue au calcul donné
par la proposition~\ref{theo.formule-ma-berk}
à l'aide de la théorie de l'intersection sur la fibre
spéciale du schéma formel~$\mathfrak X$.

Soit $x$ un point de~$X$. Soit $\xi\in\mathfrak X_s$
sa spécialisation et soit $V=\overline{\{\xi\}}$;
on suppose que $\dim(V)=n$.
Comme $X$ est sans bord, $V$ est propre sur~$\widetilde k_1$.
Soit $E=\widetilde k^1(V)$ son corps des fonctions
et soit~$F$ le corps gradué~$E\otimes_{\widetilde k^1}\widetilde k$.
Le corps gradué $\hrt x$ est fini sur
$\hrt x^1\otimes_{\widetilde k^1}\widetilde k$ de degré
$(\abs{\hrt x^\times}:\abs{k^\times})$ et $\hrt x^1$ est fini
sur~$E$; par conséquent $\hrt x$ est fini sur $F$, 
et l'on note $\lambda_{\mathfrak X}(x)$ le degré 
$[\hrt x:F]$. On peut décrire $\lambda_{\mathfrak X}(x)$ de manière non graduée
en remarquant que 
\[ \lambda_{\mathfrak X}(x) = \big(\abs{\hr x^\times}:\abs {k^\times}\big)
 [\hrt  x^1:E]. \]

\begin{defi}
Supposons que le schéma formel~$\mathfrak X$ soit propre.
On appelle \emph{cycle fondamental} du schéma formel~$\mathfrak X$
le cycle 
purement de dimension~$n$ supporté par~$\mathfrak X_s$ donné par
\[ \cycle(\mathfrak X) = \sum_{x} m(x) \df_k(x) \lambda_{\mathfrak X}(x) [V_x] ,\]
où la somme est restreinte aux points~$x$ de~$X$ tels que l'adhérence
de Zariski~$V_x$ de la spécialisation de~$x$ est de dimension~$n$.
\end{defi}

Comme l'application de spécialisation est surjective,
le coefficient de chaque composante est strictement positif.

\begin{theo}
Soit $\mathfrak X$ un $k^\circ$-schéma formel plat,
localement topologiquement de type fini,
purement de dimension relative~$n$. 
Notons~$X$ sa fibre générique, supposons-la sans bord.
Soit $\mathfrak L_1,\dots,\mathfrak L_n$ des fibrés en droites
sur~$\mathfrak X$; notons $\overline L_1,\dots,\overline L_n$ 
les fibrés en droites sur~$X$ déduits de~$\mathfrak L_i$,
munis des métriques paralinéaires strictement entières correspondantes. 

\begin{enumerate}
\item Le support de la mesure 
est contenu dans l'ensemble localement fini des points de~$X$ 
qui se spécialisent
sur le point générique d'une composante irréductible de la fibre spéciale
de~$\mathfrak X$.
\item
Soit $x$ un point de~$X$ dont la spécialisation est le point
générique d'une composante irréductible~$V$ de~$\mathfrak X_s$.
La masse du point~$x$ pour la mesure $c_1(\overline L_1)\dots c_1(\overline L_n)$
est égale à
\[ m(x) \df_k(x) \lambda_{\mathfrak X}(x)\int_Vc_1(\mathfrak L_{1,s})\cdots c_1(\mathfrak L_{n,s}).\]
\item
Si $\mathfrak X$ est propre, 
la masse totale de cette mesure est égale à 
\[
\int_{ \cycle(\mathfrak X)}c_1(\mathfrak L_{1,s})\cdots c_1(\mathfrak L_{n,s})). \]
\end{enumerate}
\end{theo}
\begin{proof}
Notons $\mu$ cette mesure $c_1(\overline L_1)\dots c_1(\overline L_n)$.
D'après la proposition~\ref{prop.MA-berk}
et la remarque \ref{rem-cl-strict} son support 
est contenu dans l'ensemble~$S$ des points~$x\in X$
tels que la dimension de définition simultanée des~$(\widetilde{L_{i,x}})^1$
soit égale à~$n$. 
Soit $x$ un point de~$X$,
soit $\xi\in\mathfrak X_s$ sa spécialisation
et soit $V$ l'adhérence réduite $\overline{\{\xi\}}$. Pour tout $i$, notons $\Lambda_i$ le fibré en droites
sur $\zr E{\widetilde k^1}$ de prémodèle $(V,\mathfrak L_i|_V)$, et $M_i$ 
l'image réciproque de $\Lambda_i$ sur $\zr F{\widetilde k}$. 
D'après le lemme \ref{lemme-premodele-formel}
le $\widetilde k^1$-schéma intègre $V$, muni des restrictions
des~$\mathfrak L_i$, est un prémodèle de $\zr {\hrt x}{\widetilde k}$, muni
des fibrés résiduels stricts $(\widetilde{L_{i,x}})^1$ ; en d'autres termes, 
l'image réciproque de $M_i$ sur $\zr{\hrt x}{\widetilde k}$ est isomorphe à 
$\widetilde{L_{i,x}}$ pour tout $i$. 
Par conséquent, la dimension de définition simultanée 
des $\widetilde {L_{i,x}}$ est inférieure ou égale à~$\dim(V)$.
En particulier, si $\dim(V)<n$, la mesure~$\mu$ est nulle au voisinage de~$x$.
Cela démontre l'assertion~\emph a).

L'assertion~\emph c) découle de l'assertion~\emph b) que
nous démontrons maintenant.
Conservons les notations $x, \xi, V,\Lambda_i, M_i$ introduites ci-dessus, en supposant de plus
que $\dim V=n$, c'est-à-dire que $\xi$ est le point générique d'une composante irréductible de 
$\mathfrak X_s$. 

On a alors
\begin{align*}
\mu(x)&= m(x) \df_k(x) \int_{\zr{\hrt x}{\widetilde k}}c_1(\widetilde{L_{1,x}})\cdots c_1(\widetilde{L_{n,x}})\\
&=m(x) \df_k(x) \lambda_\mathfrak X(x) \int_{\zr F{\widetilde k}}c_1(M_1)\cdots c_1(M_n)\\
&=m(x) \df_k(x) \lambda_\mathfrak X(x)\int_{\zr E{\widetilde k_1}}c_1(\Lambda_1)\cdots c_1(\Lambda_n)\\
&=m(x) \df_k(x) \lambda_\mathfrak X(x)\int_Vc_1(\mathfrak L_1)\cdots c_1(\mathfrak L_n),
\end{align*}
où la seconde égalité provient du~\S\ref{ss-intersection-changebasefini},
la troisième du~\S\ref{coro-intersection-graduee}, 
et la dernière de la définition du nombre d'intersection 
dans le cas non gradué. 
Ceci termine la démonstration. 
\end{proof}

\begin{prop}\label{prop.defaut-reduit}
Soit $\mathfrak X$ un $k^\circ$-schéma formel plat,
localement topologiquement de type fini,
purement de dimension relative~$n$. Soit  $x$ un point de~$X$ dont l'image~$\xi$ dans~$\mathfrak X_s$
est le point générique d'une composante irréductible~$V$. 
\begin{enumerate}
\item Si $X$ est réduit, on a $m(x)=1$.
\item Si $X$ et $\mathfrak X_s$ sont réduits,
on a $\lambda_{\mathfrak X}(x)=1$. 
De plus, $x$ est l'unique antécédent de~$\xi$, et l'on a
$\abs{\hr x^\times}=\abs{k^\times}$ et $\hrt x^1=\widetilde k^1(V)$.
\item Si $X$ et $\mathfrak X_s$ sont géométriquement réduits,
on a $\df_k(x)=1$.
\end{enumerate}
\end{prop}
\begin{proof}
L'assertion~\emph a) est évidente.

\emph b)
Supposons $X$ et $\mathfrak X_s$ réduits, et
montrons que $\lambda_{\mathfrak X}(x)=1$. 
On peut pour ce faire remplacer $\mathscr X$ par n'importe quel voisinage ouvert formel
de $\xi$ (et $X$ par la fibre générique de ce voisinage), ce qui permet
de supposer que $\mathfrak X=\Spf(A)$ où $A$ est une $k^\circ$-algèbre topologiquement
de présentation finie telle que 
$A_{\widetilde k^1}$ soit intègre. 

On a $X=\mathscr M(A_k)$. Comme $A_k$ et $A_{\widetilde k^1}$ sont réduits,
la proposition~1.1
de~\cite{BoschLutkebohmertRaynaud-1995},
assure qu'il existe une présentation 
$k\{T_1,\dots,T_n\}\to A_k$ de~$A_k$ dont la norme quotient
coïncide avec la norme spectrale de $A_k$. Il en résulte que la norme spectrale de $A_k$
prend ses valeurs dans $\abs k$, 
que $A_k^\circ=A$ et que $A_k^{\circ\circ}=k^{\circ\circ}\cdot A$;
par suite, $A\otimes_{k^\circ} \tilde k^1 \to \tilde A^1$ est un isomorphisme.
La proposition~2.4.4, (ii) et~(iii), de~\cite{berkovich1990}
entraîne alors que le point générique de~$V$ a un unique antécédent
dans~$X$, qui est donc le point~$x$, et
que $\{x\}$ est le bord de Shilov de~$A_k$.
La semi-norme sur $A_k$ qui correspond à $x$ est donc la norme spectrale
de $A_k$, 
si bien que $\abs{\hr x^\times}=\abs{k^\times}$.
Il découle alors de la proposition citée que 
$\hrt x^1=\widetilde k^1(V)$. Par suite, $\lambda_{\mathfrak X}(x)=1$.

\emph c)
Supposons $X$ et $\mathfrak X_s$ 
géométriquement réduits, et montrons que $\df_k(x)=1$. 
On se ramène comme ci-dessus au cas où $\mathfrak X=\Spf(A)$.
Soit $K$ le complété d'une clôture algébrique de~$k$;
on va appliquer la proposition~\ref{prop.calcul-defaut}
en démontrant, avec les notations de cette proposition
que l'on a $\ell=\tilde\ell$ et $\widetilde d=1$, ce qui permettra de conclure puisque $\df_k(x)=\ell\widetilde d/\widetilde \ell$. 

Notons $\mathfrak Y=\mathfrak X\otimes_{k^\circ}K^\circ$
et $Y$ sa fibre générique; on a $Y=X_K$.
Soit $W_1,\dots,W_r$ les composantes irréductibles
de~$\mathfrak Y_s$ qui dominent~$V$
et $\xi_1,\dots,\xi_r$ leurs points génériques.
Par hypothèse, on peut appliquer le résultat~\emph b)
à~$\mathfrak Y$.

Pour tout~$i$, le point~$\xi_i$  a donc un unique antécédent~$y_i$
dans~$Y$, et les~$y_i$ sont les antécédents de~$x$ dans~$Y$.
Comme la fibre~$\mathfrak X_s$ est géométriquement réduite,
l'algèbre $\widetilde k^1(V)\otimes_{\widetilde k^1}\widetilde K^1$
est le produit des $\widetilde K^1(W_i)$;
sa longueur est donc égale à~$r$.

Puisque 
$X$ et $\mathfrak X_s$ sont géométriquement réduit ils sont réduits, et
il résulte alors de~\emph b) que $\abs{\hr x^\times}=\abs {k^\times}$. 
Par conséquent 
\[ \hrt x\otimes_{\hrt k}\hrt K=
 (\hrt x^1\otimes_{\hrt k^1}\hrt K^1) \otimes_{\hrt K^1}\hrt K, \]
si bien que
l'algèbre graduée $\hrt x\otimes_{\hrt k}\hrt K$ est de longueur~$r$.
On a donc $\tilde\ell=r$.

Comme $Y$ est réduit, $\hr x\widehat\otimes_k K$ est le produit
des~$\hr {y_i}$ en vertu de la proposition~\ref{prop.calcul-defaut}, \emph a);
la longueur de cet anneau est donc égale à~$r$ ; autrement dit, $\ell=r$.

Enfin, pour tout $i$, on a $ [\hrt {y_i}^1:(\hrt K^1\cdot \hrt x^1)]=1$.
Comme $K$ est algébriquement clos, le groupe $\abs{K^\times}$
est divisible
et l'on a $ [\hrt {y_i}:\hrt K\cdot \hrt x]
=  [\hrt {y_i}^1:(\hrt K^1\cdot \hrt x^1)]$. Par suite, $\widetilde d=1$.
Cela conclut la démonstration.
\end{proof}

\begin{coro}\label{coro.cycle-fondamental-reduit}
Si $X$ et $\mathfrak X_s$ sont géométriquement réduit et si $\mathfrak X$
est propre, alors le cycle fondamental~$\cycle(\mathfrak X)$
est la classe de la fibre spéciale $[\mathfrak X_s]$.
\end{coro}

\begin{prop}\label{prop.defaut-discrete}
Supposons que la valuation de~$k$ soit discrète.
Désignons ici simplement par $\widetilde\cdot$ la réduction classique,
plutôt que par~$\widetilde\cdot^1$.
Soit $\mathfrak X$ un $k^\circ$-schéma formel 
localement topologiquement de présentation finie 
et relativement normal dans sa fibre générique~$X$,
que l'on suppose réduite. 
Soit $\xi$ un point générique de la fibre spéciale de~$\mathfrak X$.

\begin{enumerate}
\item 
Le point~$\xi$ possède un unique antécédent~$x$ dans~$X$.

\item
La multiplicité de la composante irréductible $\overline{ \{\xi\}}$ 
de $\mathfrak X_{\tilde k}$, 
c'est-à-dire la longueur de l'anneau local artinien 
$\mathscr O_{\mathfrak X_{\tilde k},\xi}$, 
est égale à l'indice $\big(\abs{\hr x^\times}:\abs{k^\times}\big)$. 

\item
Le corps $\hrt x$ s'identifie au corps $\widetilde k(\xi)$.

\item
On a $m(x)=\df_k(x)=1$ et $\lambda_{\mathfrak X}(x)$
est la multiplicité de la composante irréductible~$\overline{\{\xi\}}$
de~$\mathfrak X_{\bar k}$.

\item
Supposons que $\mathfrak X$ soit propre; alors, 
le cycle fondamental $\cycle(\mathfrak X)$ est égal à $[\mathfrak X_{\bar k}]$.
\end{enumerate}
\end{prop}
Cela démontre en particulier que les mesures définies par la théorie
de cet article
sont compatibles avec celles de~\cite{chambert-loir2006}. 

\begin{proof}
Notons $\rho$ l'application de réduction
de~$X$ dans~$\mathfrak X_{\tilde k}$.
La question est locale sur $\mathfrak X$, ce qui autorise à le supposer
affine, donc de la forme $\Spf(A)$ où $\mathfrak A$ est
une $k^\circ$-algèbre plate topologiquement de présentation finie
intégralement close dans $A_k$.
On peut également supposer que $\mathfrak X$
est équidimensionnel ; soit $d$ sa dimension.
L'espace affinoïde $X=\mathfrak X_\eta$ est alors équidimensionnel
et de dimension~$d$.
Comme il est réduit, son lieu de non-normalité~$Y$ 
est un fermé de Zariski de dimension strictement inférieure à~$d$
(\cite{ducros2009}, théorème~4.4).
Le corps $\widetilde{\hr x}$
étant de degré de transcendance~$d$ sur $\tilde k$, le point~$x$
n'appartient pas à~$Y$, donc $\xi$ n'appartient pas à~$\rho(Y)$.
Si $\tilde X$ et $\tilde Y$ désignent les réductions canoniques
des espaces affinoïdes~$X$ et~$Y$, la factorisation canonique
\[ \begin{tikzcd}
   Y \ar{r} \ar{d} & X \ar{d} \\
  \tilde Y \ar{r} & \tilde X \ar{r} & \mathfrak X_{\tilde k},
\end{tikzcd}
\]
la surjectivité des applications de réduction
et le théorème de Chevalley entraînent que $\rho(Y)$
est constructible dans~$\mathfrak X_s$ (et même fermé, car $\widetilde Y$
est fini sur $\tilde X$, qui est fini sur $\mathfrak X_s$) ; et
il est de dimension $<d$ car pour tout $y\in Y$ le corps
$\hrt y$ est de degré de transcendance
majoré par $d_k(y)$, et donc strictement inférieur à $d$, 
sur $\widetilde k_1$. 
Quitte à restreindre encore~$\mathfrak X$, 
on peut donc supposer que $X=\mathfrak X_\eta$ est intègre et normal.  
Puisque $A$ est intégralement
close dans $A_k$, l'algèbre intègre $A$ est 
également normale.  Soit $K$ son corps des fractions.

Le schéma $\mathfrak X_s$ s'identifie à $\Spec A_{\widetilde k}$ ; le point $\xi$ correspond par ce biais
à un point de $\Spec  (A)$ que l'on note encore $\xi$ et qui
est de hauteur~1 dans le schéma noethérien normal $\Spec (
A)$.  L'anneau local $\mathscr O_{\Spec (A),\xi}$ 
est donc l'anneau d'une valuation discrète de~$K$, et
il domine $k^\circ$. Soit $\abs \cdot\colon K\to \R_+$ la valuation 
d'anneau $\mathscr O_{\Spec (A),\xi}$ 
normalisée de sorte que sa restriction à $k^\circ$ 
soit la valeur absolue structurale de ce dernier. 
La composition de $\abs \cdot \colon K\to \R_+$ 
cet de l'inclusion de
$A_k$ dans $K$ est une semi-norme multiplicative bornée sur $A_k$, 
qui définit donc un point~$x$ de 
$\mathscr M(A_k)$.  L'évaluation en $x$ induit par construction 
un plongement isométrique de
$K$ dans $\hr x$ d'image dense. 
Le corps résiduel de $\hr x$ est dès
lors égal à celui de $K$, c'est-à-dire
à $\widetilde k(\xi)$, 
et son groupe des valeurs n'est autre
que $\abs{K^\times}=\abs{k^\times}^{1/e}$, où
$e$ est l'indice de ramification de 
$\mathscr O_{\Spec(A),\xi}$ sur $k^\circ$. L'indice
$e$ peut aussi s'interpréter comme la longueur
de $\mathscr O_{\Spec(A),\xi}/(\pi)$ où $\pi$ est une uniformisante de 
$k^\circ$, c'est-à-dire encore la longueur de
$\mathscr O_{\mathfrak X_s, \xi}$. Par conséquent,
la longueur de $\mathscr O_{\mathfrak X_s, \xi}$ est égale
à l'indice $(\abs{\hrt x^\times}:\abs{k^\times})$. 

Enfin, soit $y\in X$ tel que $\rho(y)=\xi$.
Par hypothèse, si un élément $a\in A$ vérifie
$a(\xi)=0$, il satisfait $\abs{a(y)}<1$. 
L'anneau de valuation de~$K$ induit par~$y$ domine donc
$\mathscr O_{\Spec(A),\xi}$. Comme ce dernier est un anneau
de valuation (discrète), les deux anneaux de valuation coïncident.
Comme les restrictions à~$k$ des valeurs absolues correspondant à~$x$
et~$y$ coïncident, ces valeurs absolues sont égales 
et l'on a $y=x$. Ceci achève la preuve de \emph a), \emph b) et \emph c). 

Démontrons~\emph d).
Comme $X$ est réduit, on a $m(x)=1$; comme la valuation de~$k$ est
discrète, le corps~$k$ est stable et l'on a $\df_k(x)=1$.
On a par ailleurs
\[ \lambda_{\mathfrak X}(x) = \big( \abs{\hr x^\times}:\abs{k^\times}\big)
 [ \hrt x:\widetilde k(\xi)]. \]
L'assertion découle donc de~\emph b) et~\emph c).

Enfin, \emph e) en résulte, compte tenu de la définition de~$\cycle(\mathfrak X)$.
\end{proof}

\section{Opérateur de Monge-Ampère relatif}

\begin{prop}\label{prop.MA-relatif}
Soit $X$ et~$Y$ des espaces analytiques sans bord, équidimensionnels,
et topologiquement séparés,
soit $p\colon Y\to X$ un morphisme génériquement plat;
on pose $m=\dim(X)$, $n=\dim(Y)$ et $d=n-m$.

Soit $(L_1,\dots,L_m)$
une famille de fibrés en droites métrisés 
paralinéaires entiers sur~$X$.
Pour tout $x\in X^\#$,  posons
$\lambda_x=m(x)\df_k(x)\left(\int_{\zr {\hrt x}{\widetilde k}}
c_1(\widetilde{L_{1,x}})\cdots c_1(\widetilde{L_{m,x}})\right)$. 

\begin{enumerate}
\item L'ensemble des $x\in X^\#$ tels que $\lambda_x\neq 0$ est une partie fermée et
discrète de $X$, et la mesure $\bigwedge c_1(L_i)$ sur $X$ est égale à $\sum
\lambda_x \delta_x$. 

\item On a l'égalité suivante de courants de bi-dimension $(d,d)$ sur $Y$ : 
\[ \bigwedge c_1(p^*L_i) = \sum_{x\in X^\#} \lambda_x \delta_{Y_x}, \]
où $\delta_{Y_x}$ est le courant d'intégration sur l'espace $\hr x$-analytique~$Y_x$.
\end{enumerate}
\end{prop}

\begin{proof}
L'assertion \emph a) est un simple rappel de la proposition~\ref{prop.MA-berk}.

Soit $\omega$ une forme lisse de bidegré~$(n-m,n-m)$ à support compact sur~$Y$.
D'après la proposition~\ref{theo.formule-projection-c1},
on a 
\[ \langle c_1(p^*L_1)\wedge \dots \wedge c_1(p^*L_m) \wedge \omega\rangle
 = \int_X u_\omega c_1(L_1)\wedge \dots \wedge c_1(L_m), \]
où le membre de droite désigne l'intégrale de la fonction~$u_\omega$
contre la mesure
$c_1(L_1)\wedge \dots\wedge c_1(L_m)$, laquelle 
est égale à $\sum_{x\in X^\#} \lambda_x \delta_x$
comme on l'a rappelé. 
Le membre de droite est alors égal à 
\[ \sum_{x\in X^\#} \lambda_x u_\omega(x) .\]
Par ailleurs, pour tout $x\in S$, 
on a $d_k(x)=m$, si bien que la fibre~$Y_x$ est de dimension~$\leq n-m$;
par platitude générique, on a $\dim(Y_x)=n-m$, si bien que 
\[ u_\omega(x) = \int_{Y_x} \omega = \langle \delta_{Y_x},\omega\rangle. \]
La proposition en résulte.
\end{proof}

\begin{coro}\label{coro.log-max-general}
Soit $X$ un espace $k$-analytique sans bord, topologiquement séparé
et équidimensionnel.  
Soit $(f_1,\dots,f_d)$ une suite de $d$~fonctions analytiques sur~$X$ 
que l'on considère comme un morphisme $f\colon X\to \A^d$. 
Pour toute suite $(t,r_1,\dots,r_d)$ de nombres réels~$>0$,
le courant~$(\ddc [\log\max(t,\abs {f_1}/r_1,\dots,\abs{f_d}/r_d))])^d$ 
est égal au courant d'intégration
sur la fibre~$f^{-1}(\eta_r)$, où $r=(tr_1,\dots,tr_d)$.
\end{coro}
\begin{proof}
Posons $\phi=\log\max(t,\abs {f_1}/r_1,\dots,\abs{f_d}/r_d))$. 
Soit $x$ un point de $X$. Supposons tout d'abord qu'il existe $i$ tel
que $\abs{f_i(x)}{r_i}$ soit différent de $t$. Dans ce cas on peut écrire  au voisinage 
de $x$ 
l'égalité $\phi=\log \max (g_1/s_1,\ldots, g_m/s_m)$ où $m\leq d$, où les $g_i$
sont des fonctions inversibles et les $s_i$ des réels strictement positifs et où
$\abs{g_i(x)}/s_i$ est un réel indépendant de $i$, que l'on note $\rho$
(chaque couple $(g_i,s_i)$ étant de la forme $(f_j, r_j)$ ou $(1,1/t)$). 
Soit $\psi$ la fonction $\phi-\log \abs{g_1}$. On a $\ddc \phi=\ddc \psi$
au voisinage de $x$, 
et $\psi$ est égale à $\log \max (1/s_1,(g_2g_1^{-1})/s_2\ldots, (g_mg_1^{-1})/s_m)$. 
Le lemme \ref{lemm.approx-max} fournit une famille $(u_\eps)$ 
de fonctions lisses et convexes sur $\R^{m-1}$ qui converge uniformément 
vers $\max (-\log s_1, x_1-\log s_2,\ldots, x_{m-1}-\log s_m)$ quand $\eps$ tend vers $0$. 
Le courant $(\ddc \psi)^d$ est alors la limite des formes 
$(\ddc u_\eps (\log \abs{g_2g_1^{-1}}, \ldots, \log \abs{g_mg_1^{-1}}))^d$. Or comme la famille
$(g_1g_i^{-1})_{2\leq i\leq m}$ est de cardinal majoré par $m-1$ et a fortiori par $d-1$, sa dimension tropicale 
est majorée par $d-1$, si bien que $(\ddc u_\eps (\log \abs{g_2g_1^{-1}}, \ldots, \log \abs{g_mg_1^{-1}}))^d$
est nul pour tout $\eps$. On a donc au voisinage de $x$ les égalités $\ddc \phi=\ddc \psi=0$. 

Supposons maintenant que $\abs{f_i(x)}=r_i$ pour tout $i$. En utilisant encore le lemme \ref{lemm.approx-max} 
on obtient une famille $(v_\eps)$ de fonctions convexes lisses sur $\R^d$ convergeant uniformément vers
$\max(\log t, x_1-\log r_1,\ldots, x_d-\log r_d)$, et le courant $\ddc \phi$ et la limite  des formes 
$(\ddc v_\eps (\log \abs{f_1}, \ldots, \log \abs{f_d}))^d$. Si la dimension tropicale de $f$ en $x$ est strictement inférieure
à $d$, le courant $(\ddc \phi)^d$ est donc nul au voisinage de $x$. L'espace $X$ étant sans bord, la dimension tropicale de $f$
en $x$ est égale au degré de transcendance de la famille
$\widetilde{f_i(x)}$. Par conséquent, $(\ddc \phi)^d$ est nul au voisinage de $x$ dès que 
$f(x)\notin S(\gm^d)$, c'est-à-dire dès que $f(x)\neq \eta_r$ puisqu'on a supposé
que $\abs{f_i(x)}=r_i$ pour tout $i$. 

Il résulte de ce qui précède que le support de $(\ddc\phi)^d$
est contenu dans $f^{-1}(\eta_r)$. Or $f$ est plat en tout point de la fibre
$f^{-1}(\eta_r)$ : cela provient du fait que $\mathscr O_{\gm^d,\eta_r}$ est un corps
(puisqu'il est artinien et que $\gm^d$ est réduit) et que $f$ est sans bord. 
On peut donc, pour démontrer le corollaire, remplacer $X$ par le lieu de platitude 
de $f$ et ainsi supposer que $f$ est plat.

Soit $L$ le fibré métrisé sur $\A^d$ de fibré sous-jacent $\mathscr O_X$, muni 
de la métrique pour laquelle $\log \norm 1^{-1}=\log \max (t, \abs{T_1}/r_1,
\ldots, \abs{T_d}/r_d)$. Par définition, 
$\ddc \phi$ est égal à $c_1(f^*L)$, si bien que 
$(\ddc \phi)^d=c_1(f^*L)^d$. Puisque le support de $(\ddc \phi)^d$
est contenu dans $f^{-1}(\eta_r)$ comme on l'a vu plus haut, la proposition 
\ref{prop.MA-relatif}
montre que \[(\ddc \phi)^d=\left(\int_{\zr{\hrt{\eta_r}}
{\widetilde k}}c_1(\widetilde{L_{\eta_r}})^d\right)\delta_{f^{-1}(\eta_r)}.\]
Il suffit donc pour conclure de montrer que
$\int_{\zr{\hrt{\eta_r}}
{\widetilde k}}c_1(\widetilde{L_{\eta_r}})^d=1$. 

Soit $s$ la section 
$\widetilde{1(\eta_r)}$
de $\widetilde{L(\eta_r)}$. Puisque $\log \norm 1^{-1}=\log \max (t, \abs{T_1}/r_1,
\ldots, \abs{T_d}/r_d)$, on a d'après le lemme \ref{lemm-max-vai}
l'égalité $v
(s)^{-1}
=\max(1, v(\widetilde{T_1(\eta_r)}), \ldots, v(\widetilde{T_d(\eta_r)})$
pour toute $v\in \zr{\hrt{\eta_r}}{\widetilde k}$.
On déduit alors de l'exemple 
\ref{exemple-o1-ZR} que $\int_{\zr{\hrt{\eta_r}}{\widetilde k}}c_1(\widetilde{L_{\eta_r}})^d=1$. 
\end{proof}

\section{Formules de Bochner--Martinelli 
et Levine}

On démontre dans ce paragraphe l'analogue non archimédien
de la formule de Bochner--Martinelli, dû à~\cite{Mihatsch-2024},
ainsi que l'analogue de la formule de Levine.

\subsection{}
Soit $X$ un $n$-espace $k$-analytique sans bord,
soit $(f_1,\dots,f_d)$ une suite régulière de fonctions holomorphes sur~$X$,
soit $Y$ le sous-espace analytique fermé de~$X$ défini par 
les équations $f_1=\dots=f_d=0$, soit $U=X\setminus Y$ le fermé complémentaire.
Par hypothèse, $Y$ est purement de codimension~$d$.
Posons $u=\log(\max(\abs{f_1},\dots,\abs{f_d}))$;
c'est une fonction continue et psh sur~$U$.
Par suite, pour tout entier naturel~$r$, on peut
former les courants $(\ddc u)^r$  et $u(\ddc u)^r$ sur~$U$
(\S\ref{ss.bt-def});
ils sont de bidegré~$(r,r)$. Comme $d_k(x)\leq d$
pour tout point $x\in Y$, ces courants 
s'étendent uniquement par zéro en des courants sur~$X$ lorsque $r<d$,
en vertu du lemme~\ref{lemm.prolongement-courant-trop}.
Nous noterons de la même façon ces prolongements.
Lorsque $r+1<d$, on a $ \ddc(u(\ddc u)^r)=(\ddc u)^{r+1}$.

\begin{theo}[Mihatsch] \label{theo.bm}
Soit $X$ un $n$-espace $k$-analytique sans bord,
soit $(f_1,\dots,f_d)$ une suite régulière de fonctions holomorphes sur~$X$,
soit $Y$ le sous-espace analytique fermé de~$X$ défini par 
les équations $f_1=\dots=f_d=0$;
posons $u=\log(\max(\abs{f_1},\dots,\abs{f_d}))$.

On a l'égalité de courants
\[ \ddc (u (\ddc u)^{d-1}) = \delta_Y. \]
\end{theo}

Lorsque $d=1$, on retrouve la formule de Poincaré--Lelong
(théorème~\ref{pcrllg}).
Comme $(\ddc u)^{d-1}$ est fermé, on voudrait interpréter le membre de gauche
de cette égalité comme $(\ddc u)^d$, mais ce produit n'est pas défini a priori
puisque $u$ n'est pas continue sur~$X$.

\begin{proof}
Pour tout nombre réel~$r\geq 0$, posons 
\[ u_r =  \log \max(r, \abs{f_1},\dots,\abs{f_d}) = \log \max(r,u). \]
Si $r>0$, $u_r$ est une fonction~\pl sur~$X$ et l'on a 
\[ (\ddc u_r)^d = \delta_{f^{-1}(\eta_r)} \]
en vertu du corollaire~\ref{coro.log-max-general}.
Soit $\omega$ une $(n-d,n-d)$-forme lisse à support compact sur~$X$.
Par définition, on a
\[ \langle \ddc( u(\ddc u)^{d-1}),\omega \rangle
=  \langle u(\ddc u)^{d-1},\ddc \omega \rangle
= \langle u(\ddc u)^{d-1},\ddc \omega|_U \rangle, \]
la forme $\ddc\omega|_U$ étant à support compact.
La fonction~$u$ est minorée par un nombre réel strictement positif
sur le support de $\ddc\omega|_U$; il existe donc un nombre réel~$r_0$
tel que $u_r=u$ pour tout nombre réel~$r$ tel que $0<r\leq r_0$.
Pour un tel~$r$, on a donc
\[ \langle \ddc( u(\ddc u)^{d-1}),\omega \rangle
= \langle u_r(\ddc u_r)^{d-1},\ddc \omega|_U \rangle, \]
et en reprenant les calculs dans l'autre sens,
\[ \langle \ddc( u(\ddc u)^{d-1}),\omega \rangle
=  \langle \ddc( u_r(\ddc u_r)^{d-1}),\omega \rangle. \]
Compte tenu de la définition des produits de courant à la Bedford--Taylor,
on a 
\[ \langle \ddc( u(\ddc u)^{d-1}),\omega \rangle
=  \langle (\ddc u_r)^{d}),\omega \rangle. \]
Par suite,  pour tout nombre réel~$r>0$ assez petit, on a
\[ \langle \ddc( u(\ddc u)^{d-1}),\omega \rangle
=\langle \delta_{f^{-1}(\eta_r)},\omega \rangle. \]
La proposition suivante assure que lorsque $r$ est assez petit, on a 
\[ \langle \ddc( u(\ddc u)^{d-1}),\omega \rangle
 =\langle \delta_Y,\omega\rangle, \]
d'où le théorème.
\end{proof}

\begin{prop}
Soit $X$ un $n$-espace analytique
et soit $f=(f_1,\dots,f_d)\colon X\to\A^d$ un morphisme
purement de dimension relative~$d$.
On suppose en outre que $(f_1,\dots,f_d)$ est une suite régulière
et on note~$Y$ le fermé analytique d'équations $f_1=\dots=f_d=0$;
il est purement de dimension~$n-d$.

Soit $\omega$ une G-forme de bidegré $(n-d,n-d)$ sur~$X$,
à support compact et à coefficients tropicalement boréliens localement bornés
(par exemple, tropicalement continus).
Pour $t$ suffisamment proche de l'origine dans~$\A^d$, on a 
\[ \langle \delta_{f^{-1}(t)},\omega\rangle =
\langle \delta_Y,\omega\rangle .\]
\end{prop}
\begin{proof}
Comme $Y$ est défini par une suite régulière de $d$~équations,
le lieu où~$X$ est~$(\mathrm S_d)$ contient~$Y$.
Puisque ce lieu est un ouvert de Zariski 
(\cite{ducros2009}, théorème~4.4), on peut supposer que $X$ est~$(\mathrm S_d)$.
La fibre en zéro de~$f$, qui est~$Y$, est purement de 
dimension~$n-d$ car $(f_1,\dots,f_d)$ est une suite régulière; 
par suite, le morphisme~$f$ est de dimension~$\leq n-d$ au voisinage de~$Y$
(\cite{ducros2007}, th. 4.9). 
D'après le Hauptidealsatz (lemme~\ref{lemm.hauptidealsatz-anal}), 
la dimension relative de~$f$ en tout point est au moins de dimension~$n-d$.
Quitte à remplacer~$X$ par un voisinage de~$Y$,
on peut donc supposer que $f$ est purement de dimension~$n-d$.

En considérant un recouvrement  du support de~$\omega$
par des domaines analytiques compacts tropicalisant~$\omega$
et en raisonnant par additivité, on peut supposer que $X$
est compact et que $\omega$ est tropicale.
Soit $g\colon X\to \gm^N$ un moment et soit $\alpha$
une $(n-d,n-d)$-forme à coefficients continus sur~$\R^N$ 
tels que $\omega=g^*\alpha$. Posons $h=(g,f)\colon X\to\gm^N\times\A^d$.

D'après le théorème~\ref{theo.tropcont}, \emph a), 
il existe un nombre réel~$r>0$
tel que pour tout $t\in D(0,r)$,
les polyèdres calibrés $g_\trop(X_t)^{(n-d)}$ sont égaux.
On a donc
\[ \langle \delta_{f^{-1}(t)}, \omega\rangle
= \int_{g_\trop(X_t)^{(n-d)}}, \alpha 
= \int_{g_\trop(X_0)^{(n-d)}}, \alpha 
= \langle \delta_{f^{-1}(0)}, \omega\rangle 
= \langle \delta_Y,\omega\rangle . \]
Ceci conclut la démonstration.
\end{proof}

\subsection{}
Passons maintenant à l'analogue de la formule de Levine.
Soit $n$ un entier naturel; posons $X=\P^n$.
Soit $p$ un entier naturels tel que $1\leq p\leq n$
et soit $f\colon X \dashrightarrow \P^{p-1}$
l'application rationnelle donnée 
par $f([x_0:\dots:x_n])=[x_0:\dots:x_{p-1}]$.
Son ouvert de définition~$U$ est 
le complémentaire du sous-espace projectif $Y\simeq\P^{n-p}$ de~$X=\P^n$
défini par $x_0=\dots=x_{p-1}=0$.

On munit le fibré en droites~$\mathscr O(1)$ sur~$\P^n$
de sa métrique formelle psh standard. 
Posons $\alpha = c_1(\overline{\mathscr O(1)})$; c'est un courant 
paralinéaire positif de bidegré~$(1,1)$ sur~$X$.

Sur l'ouvert~$U$, on peut également considérer le fibré
en droites métrisé $f^*\overline{\mathscr O(1)}$
(psh, formel)
et le courant paralinéaire positif 
de bidegré~$(1,1)$, $\beta=c_1(f^*\overline{\mathscr O(1)})$.

Sur $U$, le fibré en droites $f^*\mathscr O(1)$ est isomorphe
à la restriction de~$\mathscr O(1)$; notons~$h$ le logarithme du quotient 
de leurs deux métriques:
\[ h ([x_0:\dots:x_n]) = \log \frac{ \max(\abs{x_0},\dots,\abs{x_n})}{\max(\abs{x_0},\dots,\abs{x_{p-1}})}. \]
C'est une fonction continue paralinéaire sur~$U$.

La théorie de Bedford--Taylor permet de définir les courants positifs
$\alpha^i \beta^j$ sur~$U$; ils sont paralinéaires d'après le théorème~\ref{theo.ddcuT-paralineaire}. Posons
\[ \omega = \sum_{j=0}^{p-1} \alpha^{p-1-j} \beta^j \ ;\]
c'est un courant paralinéaire de bidimension~$(n-p+1,n-p+1)$ sur~$X$.
Comme $h$ est continue sur~$U$, on peut considérer le courant 
$ h \omega$
de bidimension~$(n-p+1,n-p+1)$ sur~$U$. 
Comme $X\setminus U$ est un fermé de Zariski de dimension~$n-p$,
il résulte de~\S\ref{ss.prolongement-courant}
que $h\omega$ se prolonge en un unique courant sur~$X$
que nous noterons~$\Lambda$.
Pour toute forme~$\eta$ de type $(n-p+1,n-p+1)$ à support compact sur~$X$, 
le support du courant $\Lambda\wedge\eta$ est contenu dans~$U$;
ainsi, $h\omega\wedge\eta|_U$ est un courant à support compact sur~$U$
et l'on a 
\[ \langle \Lambda,\eta\rangle=\langle h\omega\wedge\eta, 1\rangle
=\langle h\omega , \eta|_U\rangle. \]

\begin{theo}
On a la relation  de courants
\[ \ddc\Lambda + \delta_Y = \alpha^p \]
sur~$X$.
\end{theo}
\begin{proof}
Considérons $g\colon X'\to X$ l'éclatement de~$Y$ dans~$X$,
soit $E$ le diviseur exceptionnel et soit $U'$ l'image
réciproque de~$U$.
Le morphisme~$g$ induit un isomorphisme de~$U'=X'\setminus E$
sur~$U=X\setminus Y$.
Le morphisme $f\circ g\colon U' \to \P^{p-1}$ 
se prolonge en un morphisme $f'\colon X'\to \P^{p-1}$.
Considérons les deux fibrés en droites métrisés sur~$X'$
donnés par $\overline{\mathscr L_a}=g^*\overline{\mathscr O(1)}$
et $\overline{\mathscr L_b}=(f')^*\overline{\mathscr O(1)}$;
leurs métriques sont paralinéaires, psh; 
notons~$\alpha'$ et $\beta'$ leurs courbures.
Modulo l'isomorphisme de~$U'$ sur~$U$ induit par~$g$,
on a $\alpha'=\alpha$ et $\beta'=\beta$.
Posons également $h'=g^*h$
et $\omega'=\sum_{j=0}^{p-1}(\alpha')^{p-1-j}(\beta')^{j}$.

Les fibrés en droites sur~$U'$ correspondant sont naturellement isomorphes,
et on peut prolonger un tel isomorphisme  
en un morphisme $\rho\colon {\mathscr L_a} \to {\mathscr L_b}$
sur~$X'$ dont le lieu des zéros est~$E$.
On a $\log \norm\rho^{-1}=  h'$.

Pour toute forme lisse~$\eta$ de bidegré~$(n-p+1,n-p+1)$ sur~$X'$, 
la théorie des courants paralinéaires fournit un courant
$\omega'\wedge \eta$ sur~$X'$. 
Ce courant est une limite de formes lisses 
obtenues en remplaçant~$\omega'$ par des formes lisses
déduites d'approximations psh lisses
des métriques tropicalement convexes de~$\mathscr L_a$ et $\mathscr L_b$,
donc tropicalisées par une même famille de cartes tropicales.
Son support est ainsi contenu dans une partie paralinéaire compacte~$S$ de~$X'$;
en particulier, il ne rencontre aucun fermé Zariski  strict de~$X'$,
et est ainsi contenu dans~$U'$.
On pose alors $\langle \Lambda',\eta\rangle 
= \langle \log(h'), \omega'\wedge\eta\rangle$.
Comme on peut choisir~$S$ ne dépendant que d'un atlas tropical
tropicalisant~$\eta$, 
on voit que cette formule définit un courant~$\Lambda'$ sur~$X'$.

Sa restriction à~$U'$ s'identifie à~$h'\omega'$,
qui lui-même s'identifie à~$h\omega$ via l'isomorphisme de~$U'$ sur~$U$.
Si $\eta$ est une forme-test sur~$X$, on a donc
\[ \langle \Lambda',g^*\eta\rangle
= \langle h'\omega', g^*\eta|_{U'} \rangle = \langle h\omega, \eta|_U\rangle
=\langle \Lambda, \eta\rangle. \]
Par suite, on a $g_*\Lambda'=\Lambda$.

La formule de Poincaré--Lelong sur~$X'$ entraîne
\[ \ddc h'  + \delta_E = \alpha'-\beta'. \]

D'après la proposition~\ref{prop.PL+BT}, on a 
\[ \ddc (\Lambda') = \omega' \wedge \delta_E + (\beta'-\alpha') \wedge \omega', \]
où $\omega' \wedge\delta_E$ est défini par la théorie de Bedford--Taylor.
On a 
\[ (\beta'-\alpha')\wedge\omega' = (\beta')^p - (\alpha')^p 
 = - (\alpha')^p \]
puisque $\beta'$ provient de~$\P^{p-1}$.

Démontrons finalement l'égalité
\[ g_* (\omega'\wedge\delta_E) = \delta_Y. \]
Soit $\eta$ une forme lisse de bidegré~$(n-p,n-p)$.
Par passage à limite à partir du cas lisse, on a 
\[ \langle g_*(\omega' \wedge\delta_E), \eta \rangle
= \langle \omega'\wedge\delta_E, g^*\eta \rangle 
= \langle \omega'_E , g^*\eta|_E \rangle, \]
où 
\[ \omega'_E = \sum_{j=0}^{p-1} (\alpha'_E)^{p-1-j} (\beta'_E)^j, \]
où l'on a noté $\alpha'_E$ et $\beta'_E$ les courbures
des $\overline{\mathscr L_a}|_E$ et $\overline{\mathscr L_b}|_E$.
Il s'agit de démontrer que cette expression
vaut $\int_Y \eta$.
Les métriques de $\overline{\mathscr L_a}$
et $\overline{\mathscr L_b}$ sont limites uniformes 
de métriques lisses positives sur les fibrés  sous-jacents.
Par passage à la limite (proposition~\ref{prop.BT-continu}),
il suffit de démontrer que si $\widetilde\alpha$
et $\widetilde \beta$ sont les courbures de métriques
lisses sur les fibrés en droites~$\mathscr L_a$ et~$\mathscr L_b$,
on a 
\[ \langle \widetilde\omega|_E, g^*\eta|_E \rangle
=  \int_Y \eta,  \]
où l'on a noté
$\widetilde\omega =  \sum_{j=0}^{p-1} \widetilde \alpha^{p-1-j}\wedge \widetilde\beta^j$.
Appliquons maintenant la formule de projection du théorème~\ref{theo.formule-projection}; il vient
\[ \langle \widetilde\omega|_E, g^*\eta|_E\rangle
= \int_E g^*\eta|_E \wedge \widetilde\omega 
= \int_Y u \eta, \]
où $u$ est la fonction sur~$Y$ définie par 
$u(x)= \int_{E_x} \widetilde\omega $ lorsque $E_x$ est dimension~$p-1$,
et~$0$ sinon.
L'espace~$E$ est un fibré projectif au-dessus de~$Y$, de dimension
relative~$p-1$, la restriction du fibré en droites~$\mathscr L_b$
est son~$\mathscr O(1)$.
Pour tout $x\in Y$, 
on a
\[ u(x) = \sum_{j=0}^{p-1} (c_1(\mathscr L_a)^{p-1-j} c_1({\mathscr L_b})^j \cap [E_x]). \]
Les fibres~$E_x$ sont des espaces
projectifs, la restriction de~$\overline{\mathscr L_a}$
est triviale, celle de~$\overline{\mathscr L_b}$ est~$\mathscr O(1)$;
on a donc
\[ u(x) = (c_1(\mathscr O(1)^{p-1} \cap [\P^{p-1}]) = 1 \]
pour tout $x\in Y$.
Cela conclut la démonstration 
de l'égalité $g_*(\omega'\wedge\delta_E)=\delta_Y$.

On a $\Lambda =g_*\Lambda'$, donc
$\ddc\Lambda = \ddc(g_*\Lambda')=g_*(\ddc\Lambda')$,
qui est somme de $g_*(\omega'\wedge\delta_E)=\delta_Y$
et de $g_*(-(\alpha')^p)$.
On conclut par la formule de projection du théorème~\ref{theo.formule-projection}, en utilisant le fait que $g$ est birationnel.
\end{proof}

\section{Produits de courants de courbure (métriques formelles)}

\subsection{}
Soit $X$ un $n$-espace $k$-analytique sans bord,
soit $\bar L_1,\dots,\bar L_d$ des fibrés inversibles sur~$X$
munis de métriques paralinéaires entières. 
Le courant paralinéaire $T=c_1(\bar L_1)\wedge \dots\wedge c_1(\bar L_d)$
sur~$X$ est paralinéaire (\S\ref{ss.prod-c1-BT}). 
Dans la suite de ce paragraphe,
nous voulons le décrire plus précisément.

Au voisinage d'un point~$x$,
cette description se décompose en trois étapes, suivant
la dimension de définition analytique en~$x$ de la famille $(\overline L_j)$
de « $\Q$-fibrés métrisés ».
Si elle est~$<d$, le courant sera nul au voisinage de~$x$;
le lieu des points où elle est~$>d$ ne sera pas chargé.
Si elle est égale à~$d$, 
nous construirons 
un morphisme de germes~$f\colon (X,x)\to (Y,y)$,
où $\dim(Y)=d$, sur lequel le courant se restreint en 
un multiple du courant d'intégration sur la fibre~$f^{-1}(y)$.
Au voisinage de~$x$,
le sous-espace polyédral qui décrit le courant par rapport
à un moment donné~$g$ apparaît donc
comme le polyèdre caractéristique $\Sigma_{g|_{f^{-1}(y)}}$
du moment $g|_{f^{-1}(y)}$, muni d'un multiple
de son calibrage canonique en tant que $\hr y$-espace analytique.

\subsection{}
Soit $V$ l'ensemble des points~$x\in X$
en lesquels la dimension de définition
résiduelle de la famille $(\bar L_1,\dots,\bar L_d)$ 
est au plus égale à~$d$.

Soit aussi $W$ l'ensemble des points~$x\in X$
pour lesquels qu'il existe un entier~$q\geq 1$, puissance de l'exposant
caractéristique de~$\tilde k$,
tels que la dimension de définition
analytique de la famille $(\bar L_1^q,\dots,\bar L_d^q)$ 
est au plus égale à~$d$.

Par définition, $W$ est ouvert dans~$X$
et l'on a $V\subset W$ 
d'après la proposition~\ref{prop.dim-def-analytique}.
(Cette proposition montre aussi que l'on peut prendre $q=1$
si le corps~$\tilde k$ est parfait.)
D'autre part, $X\setminus V$ est contenu
dans une partie faiblement analytique de dimension~$<n-d$
(proposition~\ref{prop.dimdefn}),
donc tropicalement $\mathscr H^{n-d}$-négligeable.
Il en résulte que $X\setminus W$ est tropicalement
$\mathscr H^{n-d}$-négligeable; en particulier,
il ne charge pas le courant $c_1(\bar L_1)\wedge\dots\wedge c_1(\bar L_d)$.
Autrement dit, pour toute forme lisse~$\alpha$ de type~$(n-d,n-d)$
sur~$X$, à support compact, la mesure~$T\wedge\alpha$
ne charge pas $X\setminus W$, de sorte que
\[ \langle T,\alpha\rangle = \int_X T\wedge\alpha = \int_W T \wedge\alpha.\]
Cela ramène à calculer la restriction à~$W$ du courant~$T$.
Cela prouve également que $T$ est positif si et seulement si $T|_W$ l'est.

\subsection{}
Soit $x\in W$. Par hypothèse, il existe
un morphisme plat (et sans bord)
de germes séparés $f\colon (X,x)\to (Y,y)$,
où $\dim(Y)\leq d$,
un entier~$q$, des fibrés inversibles 
$\overline M_1,\dots,\overline M_d$ sur~$Y$ 
munis de métriques paralinéaires entières
tels que $\overline L_i^{q} \simeq f^*\overline M_i$
pour tout~$i$. 
Notons que comme ici $(X,x)$ est sans bord et les germes en jeu séparés, 
le morphisme~$f$ est en fait automatiquement sans bord, et $(Y,y)$ aussi. 

Si $\dim_y(Y)<d$, il résulte de la remarque~\ref{rema.courants-chgt-base-zero}
que le courant $T$ est nul au voisinage de~$x$.

Supposons maintenant que l'on a $\dim_y(Y)=d$.
D'après la proposition~\ref{prop.MA-relatif},
et quitte à restreindre~$Y$, le courant~$T$ coïncide avec le courant
$q^{-d}\lambda \delta_{f^{-1}(y)}$ au voisinage de~$x$,
où 
\[ \lambda = m(x) \df_k(x) \left( \prod_{j=1}^d c_1( \widetilde M_i(x)) \cap [ \widetilde{(X,x)}] \right).
\]

\begin{theo}\label{theo.prod-courants-paralineaires-local}
Soit $X$ un espace analytique sans bord, topologiquement séparé.
Soit $\overline {L_1},\dots,\overline{L_d}$ des fibrés
en droites sur~$X$ munis de métriques paralinéaires entières.
Soit $x$ un point de~$X$.
Pour que le courant $c_1(\overline L_1)\wedge\dots\wedge c_1(\overline L_d)$
soit positif (\resp nul) au voisinage de~$x$, il faut et il suffit que
le produit des classes de Chern
$c_1(\widetilde L_{1,x})\dots c_1(\widetilde L_{d,x})$
sur~$\widetilde{(X,x)}$
soit numériquement effective (\resp numériquement triviale).
\end{theo}
Rappelons cette condition (\S\ref{ss-numeff-nongradue}): 
choisissons un prémodèle~$V$ 
de~$\widetilde{(X,x)}$ et des fibrés en droites $M_1,\dots,M_d$
sur~$V$ qui sont des prémodèles des~$\widetilde{L_j}$;
il s'agit de dire que pour tout sous-variété irréductible~$Z$ 
de codimension~$m$ de~$V$, le degré de la classe de zéro-cycle
$(c_1(M_1)\dots c_1(M_d)\cap [Z])$ sur~$V$ 
est effective (\resp numériquement triviale).
Il est classique qu'elle ne dépend pas du choix des prémodèles
(\cite{Lazarsfeld-2004}, exemple~1.4.4).

\begin{proof}
Comme un courant ou une classe numérique de cycle sont nuls 
si et seulement s'ils sont à la fois positif et négatif,
il suffit de traiter l'assertion relative à la positivité.

Notons $m$ le plus petit entier tel qu'il existe un entier~$q\geq 1$
de sorte que la dimension de définition analytique en~$x$ de la famille
$(\overline L_1^q,\dots,\overline L_d^q)$ soit égale à~$m$.
Par homogénéité, on se ramène au cas où $q=1$.
Considérons un morphisme plat $f\colon (X,x)\to (Y,y)$, où $\dim(Y)=m$,
et des fibrés inversibles $\overline M_1,\dots,\overline M_d$
sur~$Y$ munis de métriques paralinéaires entières
tels que $\overline L_i\simeq f^*\overline M_i$ pour tout~$i$.
Notons aussi $T$ le courant considéré.

Si $m<d$, le courant~$T$ est nul, de même
que le produit de classes de Chern
$c_1(\widetilde L_1(x))\dots c_1(\widetilde L_d(x))$,
puisqu'il existe un prémodèle de dimension~$<d$.

Supposons $m=d$. 
On a alors vu qu'au voisinage de~$x$,
le courant~$T$ est de la forme $\lambda\delta_{f^{-1}(y)}$.
où $\lambda$ est un multiple strictement positif
du nombre d'intersection
$(c_1(\widetilde M_1(y))\dots c_1(\widetilde M_d(y)) \cap [\widetilde{(Y,y)}])$,
qui est positif ou nul si et seulement si
le produit de classes de Chern
$(c_1(\widetilde M_1(y))\dots c_1(\widetilde M_d(y))$ est numériquement
effectif.
Tout prémodèle des $\widetilde M_j(y)$ fournit un prémodèle
des~$\widetilde L_j(x)$, d'où l'équivalence dans ce cas
car le courant d'intégration $\delta_{f^{-1}(y)}$ est positif 
et non nul au voisinage de~$x$.

Il reste à traiter le cas où $m>d$.
À toute famille finie $\Phi=(\phi_i)_{i\in I}$ d'éléments de $\hrt x^\times$,
on associe un recouvrement affine~$\mathscr U_\Phi$ 
de l'espace de Zariski--Riemann
$\widetilde{(X,x)}$ formé des parties de la forme 
$\mathscr U_{\Phi}^\eps=\P\{ (\phi_i^{\eps_i})_i\}$, 
où $\eps=(\eps_i)$ parcourt $\{-1,1\}^I$.
Choisissons une telle famille~$\Phi$ telle que pour tout~$\eps$
et tout~$j$, le fibré résiduel~$\widetilde L_{j,x}$ possède une
section inversible homogène $\sigma_{j,\eps}$ sur~$\mathscr U_\Phi^\eps$.
Pour tout~$i$, choisissons un élément~$f_i\in\mathscr O_{X,x}^\times$
tel que $\widetilde{f_i(x)}=\phi_i$.
Pour tout~$j$, choisissons une section inversible~$s_{j,\eps}$
du fibré en droites~$L_j$ au voisinage de~$x$
telle que $\widetilde{s_{j,\eps}(x)}=\sigma_{j,\eps}$.

Choisissons alors un voisinage affinoïde~$V$ de~$x$ sur lequel
les fonctions~$f_i$ et les sections~$s_{j,\eps}$ sont définies et inversibles;
pour tout~$\eps\in\{\pm1\}^I$, notons~$V_\eps$ le domaine affinoïde de~$V$
défini par les inégalités~$\abs{f_i^{\eps_i}}\leq \abs{f_i(x)}$.
On peut supposer que pour tout~$j$ et tout~$\eps$,
$\norm{s_{j,\eps}}= \norm{s_{j,\eps}(x)}$ sur~$V_\eps$.

Pour tout~$\eps$, la réduction~$\widetilde V_\eps$ de~$V_\eps$
est un $\tilde k$-schéma gradué de type fini; ces schémas se recollent
en un $\tilde k$-schéma gradué de type fini~$\mathscr S$.
Les applications de réduction $V_\eps\to \widetilde V_\eps$ se recollent
en une une application de réduction $\pi\colon V\to \mathscr S$,
anticontinue. Pour tous $s\in \mathscr S$, l'adhérence de
$\pi^{-1}(s)$ dans $V$ est la réunion des $\pi^{-1}(t)$, où $t$ parcourt
l'ensemble des générisations de $s$ dans $\mathscr S$ 
(cela se vérifie sur chacun des $V_\eps$, pour lesquels on applique
le lemme \ref{lemme-adherence-tube}). 

Pour tout~$j$ et tout~$\eps$, la section $s_{j,\eps}$ du fibré en droites 
métrisé $L_j|_{V_\eps}$
est de norme constante. 
Lorsque $\eps$ varie, les réductions graduées des 
$\mathscr O(V_\eps)$-modules normés  $L_{j,\eps}(V_\eps)$ 
se recollent en un fibré en droite gradué $\mathscr L_j$ sur~$\mathscr S$.

Pour tout $y\in V$,
l'adhérence de Zariski $\overline{\{\pi(y)\}}$ de~$\pi(y)$ 
est un prémodèle de~$\widetilde{(X,y)}$.
L'espace de Zariski--Riemann $\widetilde{(V,y)}$ 
est égal à $\P_{\hrt y/\tilde k(\pi(y))}\{\overline{\pi(y)}\}$;
en particulier, si $y\in\mathring V$, 
on a $\widetilde{(X,y)}=\widetilde{(V,y)}=
\P_{\hrt y/\tilde k(\pi(y))}\{\overline{\pi(y)}\}$,
ce qui entraîne que $\overline{\pi(y)}$ est propre.
Pour tout~$j$,
la restriction de~$\mathscr L_j$ à~$\overline{\{\pi(y)\}}$
est un prémodèle de~$\widetilde{L_{j,y}}$.
(Pour la variante formelle de ce raisonnement, voir la
démonstration du lemme~\ref{lemme-premodele-formel}.)

En appliquant ce qui précède, on en déduit que le produit des classes de Chern
$c_1(\widetilde L_{1,x})\dots c_1(\widetilde L_{d,x})$
est numériquement effectif sur~$\widetilde{(X,x)}$ 
si et seulement si l'on a
$c_1(\mathscr L_1|_W)\dots c_1(\mathscr L_d|_W)\geq 0$
pour toute sous-variété irréductible~$W$ de~$\overline{\pi(x)}$ 
telle que $\dim(W)=d$.

Supposons que ce soit le cas et soit $U=\pi^{-1}(\overline{\pi(x)})$.
C'est un voisinage ouvert de~$x$ dans~$V$
et donc un voisinage de~$x$ dans~$X$.
Soit $y$ un point de~$U\cap\mathring V$ 
tel que $W=\overline{\pi(y)}$ est de dimension~$\leq d$.
Comme $W$, muni de la restriction des~$\mathscr L_j$, est un prémodèle
des $\widetilde L_{j,y}$ sur~$\widetilde{(X,y)}$, 
le produit des classes de Chern 
$c_1(\widetilde L_{1,y})\dots c_1(\widetilde L_{d,y})$ 
est numériquement effectif, par l'hypothèse faite.
D'après le cas « $m\leq d$ » déjà traité, le courant~$T$
est positif sur un voisinage ouvert~$U_y$ de~$y$ dans~$\mathring U$.
Le complémentaire dans~$\mathring U$ de la réunion des~$U_y$
est une partie  fermée de~$\mathring U$, 
c'est l'image réciproque par~$\pi$ de l'ensemble des points de~$\mathscr S$
tels que $\dim(\overline{y})> d$.
D'après la la proposition~\ref{prop.tube-faiblement-analytique}
appliquée à chacun des~$V_\eps$, c'est une partie faiblement
analytique de~$\mathring U$ de dimension~$<n-d$.
D'après l'exemple~\ref{exem.fan-neg}, ce fermé est tropicalement
$\mathscr H^{n-d}$-négligeable.
Par suite, le courant~$T$ est positif sur~$\mathring U$.

Inversement, supposons qu'il existe un voisinage ouvert~$U$
 de~$x$ dans~$\mathring V$ tel que $T|_U$ soit un courant positif.
Soit $W$ une sous-variété irréductible de dimension~$d$ de~$\overline{\pi(x)}$
et soit $w$ son point générique. Alors, l'adhérence de~$\pi^{-1}(w)$ 
contient~$x$; en particulier, il existe un point $y\in \pi^{-1}(w)\cap U$.
Par hypothèse, le courant~$T$ est positif au voisinage de~$y$
et les~$\mathscr L_j|_W$ fournissent un prémodèle des~$\widetilde L_{j,y}$.
D'après le cas « $m\leq d$ » traité, le nombre d'intersection
$(c_1(\mathscr L_1|_W)\dots c_1(\mathscr L_d|_W)\cap [W])$
est positif ou nul. Comme cela est vrai pour tout~$W$,
cela démontre que le produit des classes de Chern
$c_1(\widetilde L_{1,x})\dots c_1(\widetilde L_{d,x})$ 
sur $\widetilde{(X,x)}$ est numériquement effectif.
\end{proof} 

\begin{coro}\label{coro.PL-prod-positif}
Soit $\overline L_1,\dots,\overline L_d$ des fibrés en droites sur~$X$
munis de métriques paralinéaires.
Si les courants $c_1(\overline L_1),\dots,c_1(\overline L_d)$
sont positifs, il en est de même de leur produit
$c_1(\overline L_1)\wedge\dots\wedge c_1(\overline L_d)$.
\end{coro}
\begin{proof}
Quitte à remplacer les~$\overline{L_i}$ par une puissance convenable,
on suppose qu'ils sont munis de métriques paralinéaires entières.
En effet, sous cette hypothèse, les classes de Chern
$c_1(\widetilde L_{j,x})$ des fibrés résiduels
sont numériquement effectives, pour tout $x\in X$.
Il résulte du théorème de Kleiman qu'il en est alors
de même de leur produit (voir \cite{Lazarsfeld-2004}, exemple~1.4.16).
\end{proof}

\begin{coro}\label{coro.image-inverse-PL-positif}
Soit $Y$ un espace analytique sans bord sur une extension complète de $K$ 
et soit $f\colon Y\to X$ un $K$-morphisme. 
Soit $\overline L_1,\dots,\overline L_d$ des fibrés en droites sur~$X$
munis de métriques paralinéaires.
Si le courant $c_1(\overline L_1)\wedge\dots\wedge c_1(\overline L_d)$
sur~$X$ est positif,
il en est de même du courant
$c_1(f^*\overline L_1)\wedge\dots\wedge c_1(f^*\overline L_d)$
sur~$Y$.
\end{coro}
\begin{proof}
Quitte à remplacer les~$\overline{L_i}$ par une puissance convenable,
on suppose qu'ils sont munis de métriques paralinéaires entières.
Pour tout $i$ posons $\overline M_i=f^*\overline L_i$
Soit $y\in Y$ et soit $x$ son 
image sur $X$. 
Notre hypothèse sur les 
$\overline L_i$ $x\in X$ et le théorème ci-dessus assurent que 
le produit des classes de Chern
$ c_1(\widetilde L_{1,x})\dots c_1(\widetilde L_{d,x}) $
est numériquement effectif.
Il en va alors de même de
$c_1(\widetilde M_{1,y})\dots c_1(\widetilde M_{d,y})$
(voir \ref{ss-numeff-nongradue}), et l'on conclut en utilisant à nouveau le théorème. 
\end{proof}

\begin{coro}
Soit $f\colon Y\to X$ un morphisme d'espaces analytiques sans bord,
soit $u$ une fonction
paralinéaire\emph{harmonique} sur~$X$.
Alors $u\circ f$ est harmonique.
\end{coro}
\begin{proof}
La fonction~$u$ est harmonique si le fibré en droites trivial
muni de la métrique paralinéaire pour laquelle $\norm 1=\exp(-u)$ est 
à courbure nulle. Le corollaire découle alors du corollaire précédent.
\end{proof}

\begin{coro}\label{coro.PL-psh-c1pos}
Soit $X$ un espace analytique sans bord, soit $\overline L$
un fibré en droites sur~$X$ muni d'une métrique paralinéaire entière.
Soit $x$ un point de~$X$.
Les conditions suivantes sont équivalentes:
\begin{enumerate}\def\theenumi{\roman{enumi}}\def\labelenumi{(\theenumi)}
\item Le fibré métrisé $\overline L$ est psh au voisinage de~$x$;
\item Le courant $c_1(\overline L)$ est positif au voisinage de~$x$;
\item Le fibré en droites résiduel~$\widetilde L(x)$
est numériquement effectif.
\end{enumerate}
\end{coro}
\begin{proof}
L'implication (i)$\Rightarrow$(ii) résulte
de la proposition~\ref{prop.psh-courant};
l'équivalence (ii)$\Rightarrow$(iii) est le théorème~\ref{theo.prod-courants-paralineaires-local} et
l'implication (iii)$\Rightarrow$(i) est la proposition~\ref{prop.nef-vers-psh}.
\end{proof}

\begin{coro}\label{coro.psh-harmo}
On suppose que $X$ est sans bord et que la valeur absolue de~$k$ 
est non triviale.
Soit $\tau\colon X\to\R^n$ une fonction strictement paralinéaire 
à composantes harmoniques.
Soit $u$ une fonction convexe continue sur~$\R^n$.
La fonction $\tau^* u$ sur~$X$ est psh.
\end{coro}
\begin{proof}
Nous allons démontrer le corollaire en trois étapes.

\emph a) Supposons d'abord que $u$ est lisse.
Comme $\tau$ est harmonique, le courant $\ddc(\tau^*u)$ coïncide 
avec le courant associé à la G-forme $\tau^*\ddc (u)$, qui est positive
car $u$ est convexe et lisse.

\emph b) Supposons maintenant que $u$ est $\Q$-paralinéaire.

En remplaçant~$\tau$ par un multiple~$a\cdot\tau$
et $u$ par $x\mapsto u(a^{-1} \cdot x)$,
on se ramène au cas où $\tau$ est paralinéaire entière.
Il existe alors une suite~$(u_n)$ de fonctions lisses convexes
sur~$\R^n$ qui converge localement uniformément vers~$u$.
Alors $\tau^* u_n$ converge localement uniformément vers~$\tau^* u$
sur~$X$, de sorte que la suite de courants $(\ddc(\tau^*u_n))$ converge 
vers le courant $\ddc(\tau^* u)$.
D'après le cas~\emph a), les courants $\ddc(\tau^*u_n)$ sont positifs. 
Par suite, le courant $\ddc(\tau^*u)$ est positif.
Le corollaire~\ref{coro.PL-psh-c1pos} entraîne alors que $u$ est~psh. 

\emph c) Traitons maintenant le cas général.
L'assertion est locale sur~$X$; soit $x$ un point de~$X$
et soit $W$ un voisinage compact de~$x$ dans~$X$.

Soit $\eps>0$. 
Il existe une décomposition simpliciale de~$\R^n$ 
dont les sommets sont à coordonnées entières.
Dilatons-la d'un facteur assez petit dans~$G$
de sorte que $u$ ne varie pas de plus de~$\eps$ 
sur chaque simplexe rencontrant~$\tau(W)$.

Soit $P$ l'enveloppe convexe de la réunion des simplexes 
qui rencontrent~$\tau(W)$; c'est une $\Q$-cellule compacte de~$\R^n$
munie d'une décomposition simpliciale~$\mathscr C$ à sommets dans~$\Q^n$.

L'ensemble des fonctions continues sur~$P$ qui sont $\R$-affines
sur chaque cellule de~$\mathscr C$ est un espace affine~$\mathscr F$;
il s'identifie à $\R^{\mathscr C_0}$. Il est muni de la $\Q$-structure
$\Q^{\mathscr C_0}$.
Les fonctions $f\in\mathscr F$ qui sont convexes 
en constituent un cône convexe fermé $\mathscr F_{\mathrm {conv}}$.
Démontrons que ce cône est défini par 
un nombre fini de formes linéaires à coefficients dans~$\Q$
sur~$\R^n$.

Soit $f\in\mathscr F$; pour toute $n$-cellule $C\in\mathscr C$,
notons $f_C$ l'unique fonction affine sur~$\R^n$
qui coïncide avec~$f$ sur~$C$.
Pour que $f$ soit convexe, il faut et il suffit que
pour tout sommet $a\in\mathscr C_0$  et toute cellule $C\in\mathscr C_n$,
on ait $f_C(a)\leq f(a)$.
(En effet, on alors $f_C(x)\leq f(x)$ pour tout $x\in P$ et tout $C\in\mathscr C_n$, d'où $\sup(f_C)\leq f$, et l'inégalité inverse est évidente.)
Soit $(\lambda_i)$ la famille normalisée des coordonnées barycentriques
de~$a$ par rapport aux sommets~$(x_i)$ de~$C$; 
ce sont des nombres
rationnels car la décomposition simpliciale~$\mathscr C$
est dilatée d'une décomposition à sommets dans~$\Z^n$;
on a de plus $f_C(a)=\sum \lambda_i f_C(x_i)$, et $f_C(x_i)=f(x_i)$
pour tout~$i$,
de sorte que la relation $f_C(a)\leq f(a)$ s'écrit
$f(a)- \sum \lambda_i f(x_i)\geq 0$.

Cela prouve que 
$\mathscr F_{\mathrm{conv}}$ est un cône polyédral
$\Q$-rationnel de~$\mathscr F$, de sommet~$0$.

L'interpolation affine par morceaux~$u_P$ de~$u|_P$ relative
à la décomposition simpliciale~$\mathscr C$
est encore convexe, car $u$ est convexe.
Elle définit un point \emph{réel} de~$\mathscr F_{\mathrm{conv}}$.
Par suite, il existe des points rationnels de $\mathscr F_{\mathrm{conv}}$
arbitrairement proches de~$u|_P$ aux sommets de~$\mathscr C$.
En particulier, il existe une fonction~$u_\eps$ sur~$P$,
continue, convexe, $\Q$-affine sur chaque cellule de~$\mathscr C$,
prenant des valeurs rationnelles aux sommets de~$\mathscr C$,
telle que $\abs{u_\eps-u_P}\leq \eps$ sur~$P$.
Puisque la décomposition simpliciale~$\mathscr C$ avait choisie 
assez fine pour que 
$\abs{u_P-u}\leq \eps$ sur~$P$, 
on a $\abs{u_\eps-u}\leq 2\eps$ sur~$P$.

La fonction $\tau^* u_\eps$ sur~$W$ est alors paralinéaire 
et approche~$\tau^* u$ à~$2\eps$-près.
D'après~\emph b), elle est psh sur~$\mathring W$.

En faisant tendre~$\eps$ vers~$0$, on en déduit que $\tau^*u$
est psh sur~$\mathring W$. Cela conclut la démonstration.
\end{proof}

\begin{exem}\label{exem.courants-produits-formel}
(Dans cet exemple, la notation~$\widetilde\cdot$ désigne
la réduction usuelle, non graduée.)
Soit $\mathfrak X$ un schéma formel sur~$k^\circ$,
plat, localement topologiquement de présentation finie,
à fibre spéciale équidimensionnelle,
et soit $\mathfrak L_1,\dots,\mathfrak L_d$ une suite de fibrés
en droites sur~$\mathfrak X$.
La fibre générique~$\mathfrak X_\eta$ de~$\mathfrak X$
est un espace $k$-analytique équidimensionnel.
Notons $X$ l'intérieur de~$\mathfrak X_\eta$:
c'est l'image réciproque par l'application de réduction~$\pi\colon\mathfrak X_\eta\to\mathfrak X_s$
des points de~$\mathfrak X_s$ dont l'adhérence est propre sur~$\tilde k$.
Les fibrés en droites~$\mathfrak L_j$
induisent des fibrés en droites munis de métriques strictement~\pl sur~$X$,
notons-les $\overline{L_j}$.

\begin{enumerate}
\item
Expliquons comment déduire 
du théorème~\ref{theo.prod-courants-paralineaires-local}
et de la description du fibré résiduel d'un fibré formel
(\S\ref{sss.fibre-residuel-formel})
une description du \emph{support} du courant paralinéaire
$c_1(\overline L_1)\wedge\dots\wedge c_1(\overline L_d)$ sur~$X$.

Soit $x$ un point de~$X$.
Soit $V_x$ la $\tilde k$-variété $\overline{\{\pi(x)\}}$;
elle est irréductible et propre.
Par définition, pour tout~$j$, la restriction 
$\mathfrak L_j|_{V_x}$
à~$V_x$
du fibré formel~$\mathfrak L_j$ est un prémodèle
du fibré en droites résiduel $\widetilde L_{j,x}$.

D'après le théorème~\ref{theo.prod-courants-paralineaires-local}
et la remarque qui le suit, 
le courant 
$c_1(\overline L_1)\wedge\dots\wedge c_1(\overline L_d)$ 
est nul au voisinage de~$x$ si et seulement si
la classe de cycle numérique
$c_1(\mathfrak L_1|_{V_x})\cap \dots \cap c_1(\mathfrak L_d|_{V_x})\cap[{V_x}]$
sur~$V$ est numériquement triviale.
Autrement dit, le support du courant 
$c_1(\overline L_1)\wedge\dots\wedge c_1(\overline L_d)$ sur~$X$
est l'ensemble des points~$x$ tels que la classe
$c_1(\mathfrak L_1|_{V_x})\cap \dots \cap c_1(\mathfrak L_d|_{V_x})\cap[V_x]$
ne soit pas numériquement triviale.

De manière duale, il en résulte que le complémentaire du support de ce courant
est la réunion des $\pi^{-1}(W)$, où $W$ parcourt l'ensemble
des sous-variétés irréductibles propres de~$\mathfrak X_s$ telles
que 
$c_1(\mathfrak L_1|_{W})\cap \dots \cap c_1(\mathfrak L_d|_{W})\cap[{W}]$
soit numériquement triviale.
(En effet, si $x\in\pi^{-1}(W)$, on aura $V_x\subset W$.)
Comme $\pi$ est anticontinue, observons que $\pi^{-1}(W)$ est
une partie ouverte de~$X$, 
pour toute sous-variété irréductible propre~$W$ de~$\mathfrak X_s$,
de sorte que la réunion indiquée est effectivement ouverte.

La réunion des sous-variétés irréductibles 
propres~$W$ de~$\mathfrak X_s$
telles que 
$c_1(\mathfrak L_1|_{W})\cap \dots \cap c_1(\mathfrak L_d|_{W})\cap[{W}]$
soit numériquement triviale
semble difficile à décrire  en général.
Toutefois, 
lorsque chaque composante irréductible~$W$ de~$\mathfrak X_s$ est propre 
et que les restrictions des $\mathfrak L_j$ y induisent des fibrés  amples, 
cet ensemble est la réunion des sous-variétés irréductibles
de~$\mathfrak X_s$ qui sont de dimension~$<d$.

Par exemple, le support du courant $c_1(\overline{\mathscr O(1)})^d$ 
sur $\P^n_k$, le fibré en droites 
$\mathscr O(1)$ étant muni de sa métrique formelle canonique,
est l'ensemble des points de~$\P^n_k$ dont la spécialisation sur~$\P^n_{\widetilde k}$ a une adhérence de dimension~$\geq d$.

\item
On peut mener une analyse similaire pour étudier la positivité
du courant 
$c_1(\overline L_1)\wedge\dots\wedge c_1(\overline L_d)$ 
au voisinage d'un point~$x$ de~$X$.
Avec les mêmes notations, ce courant est positif au voisinage
de~$x$ si et seulement si
la classe
$c_1(\mathfrak L_1|_{V_x})\cap \dots \cap c_1(\mathfrak L_d|_{V_x})\cap[V_x]$
est numériquement effective. De manière duale,
l'image réciproque de la réunion des sous-variétés irréductibles propres~$W$
telles que 
$c_1(\mathfrak L_1|_{W})\cap \dots \cap c_1(\mathfrak L_d|_{W})\cap[{W}]$
soit numériquement effective
est le plus grand ouvert de~$X$ sur lequel
le courant 
$c_1(\overline L_1)\wedge\dots\wedge c_1(\overline L_d)$ 
est positif.

Dans le cas particulier où toute composante irréductible de~$\mathfrak X_s$ 
est propre, on en déduit que le courant
$c_1(\overline L_1)\wedge\dots\wedge c_1(\overline L_d)$ 
est positif si et seulement si, pour toute composante irréductible~$W$
de~$\mathfrak X_s$,
la classe de cycle
$c_1(\mathfrak L_1|_{W})\cap \dots \cap c_1(\mathfrak L_d|_{W})\cap[W]$ sur~$W$
est numériquement effective.
\end{enumerate}
\end{exem}

\section{Détecter la positivité par une sous-variété}

En géométrie complexe, le caractère psh (ou pluriharmonique)
d'une métrique hermitienne sur un fibré en droites 
se définit par restriction à toute courbe, mais est équivalent
à la positivité du courant de courbure.
Ce n'est pas ainsi que nous avons procédé dans le cas non archimédien.

Dans ce paragraphe, on donne un résultat de ce genre,
qui détecte la stricte positivité d'une métrique paralinéaire
sur un fibré en droites au voisinage d'un point d'Abhyankar,
par restriction à une courbe convenable, éventuellement
sur une extension valuée complète du corps de base.

\begin{lemm}\label{deggen} 
Soit $K$ un corps
gradué, soit $q\colon V\ra Z$ un 
morphisme surjectif de $K$-schémas gradués propres et intègres,
soit $d$ sa dimension relative générique. 
Soit $(L_1,\dots,L_d)$ une famille finie de fibrés en droites sur~$V$ ;
on suppose que $c_1(L_1)\dots c_1(L_d)$ est numériquement effectif 
et qu'il existe un point $z$ de~$Z$
de corps gradué résiduel fini sur $K$
et une composante irréductible~$H$ de~$q^{-1}(z)$ 
telle que $\int_H c_1(L_1)\dots c_1(L_d) >0$. 
Alors, le degré de~$c_1(L_1)\dots c_1(L_d)$ sur la fibre
générique de~$q$ est strictement positif.
\end{lemm}

\begin{proof}
Le lemme \ref{lemm-composantes-graduees}
permet de se ramener, par  extension des scalaires à $K(T/\rho)$ 
pour un polyrayon $\rho$ convenable, au cas non gradué. 
Notons~$\eta$ le point générique de~$Z$; la fibre générique de~$q$
est une $K(\eta)$-variété~$W$ et il s'agit de démontrer
que $\int_W c_1(L_1|_W)\dots c_1(L_d|_W)>0$.

Quitte à étendre les scalaires à une clôture algébrique $\bar K$ de~$K$, 
à remplacer~$V$ par une composante irréductible de~$V_{\bar K}$,
puis $Z$ par l'image de cette composante sur $Z_{\bar K}$ et enfin $H$
par une composante irréductible de~$H_{\bar K}$ contenue dans~$V_{\bar
K}$, on peut supposer que $K$ est algébriquement clos.

Posons $n=\dim(V)$ et $m=\dim(Z)$, de sorte que $d=n-m$.
Soit $U$ un ouvert non vide de $Z$ au-dessus duquel $q$ est plat. 

Soit $Z'$ un voisinage affine de~$z$ dans $Z$ et soit $V'$ l'image
réciproque de~$Z'$ sur $W$. Le lemme de normalisation de Noether
fournit un morphisme fini et surjectif $\pi\colon Z'\to \A^{d}_K$ ; on note
$p = \pi\circ q \colon  V'\to \A^{m}_K$ la flèche composée.
Soit $D$ une droite de $\A^{m}_K$ passant par le $K$-point $\pi(z)$ qui ne soit
pas contenue dans le fermé strict $\pi(Z'\setminus U)$ de $\A^{m}_K$;
il en existe puisque $K$ est infini. 

Soit $H'$ une composante irréductible de $p^{-1}(D)$ contenant $H$;
démontrons que $p(H')=D$. 
Comme $D$ peut
être localement (et même globalement !) définie par $m-1$ équations,
il découle du \emph{Hauptidealsatz} que $H'$ est de dimension au moins
égale à $n-(m-1)=d+1$. 
Comme $H$ est de dimension $d=n-m$ et $\pi$ est fini,
les morphismes~$p$ et~$q$
sont de dimension relative~$d$ au point générique de~$H$;
il en est donc de même de la restriction de~$p$ à~$H'$,
car $H'$ contient~$H$.
Puisque $H'$ est de dimension~$d+1$,
il n'est pas entièrement contenu dans $p^{-1}(\pi(z))$ qui est
de dimension~$d$ en au moins un point.
Puisque $p(H')$ est par ailleurs contenu dans la
droite $D$, il s'ensuit que $p(H')=D$.

Munisons $H'$ et $D$ de leurs structures réduites.
Alors, $D$ est normale (c'est une droite),
$H'$ est intègre et le morphisme $p\colon H'\to D$ est propre
et surjectif. Par conséquent, $H'$
est plat sur $D$ et en particulier purement de dimension relative $n-d$ 
sur $D$ (puisqu'il existe par ce qui précède
au moins un point en lequel il est de dimension relative $n-d$ sur $D$).

Comme $c_1(L_1)\dots c_1(L_d)$ est nef et comme  $\int_H P>0$, 
le degré 
$\int_{(\pi|_{H'})^{-1}(p(z))}c_1(L_1)\dots c_1(L_d)$ 
est strictement positif ; dès lors,
le degré de $c_1(L_1)\dots c_1(L_d)$ sur n'importe quelle fibre 
de $p\colon H'\to D$ est strictement positif
(« conservation du nombre », cf.~\cite[théorème~10.2]{fulton98}).

Choisissons $t\in  D$ n'appartenant pas à $p(H'\setminus U)$.
On a donc $p^{-1}(t)\subset U$ et 
la fibre $p|_{H'}^{-1}(t)$ est purement de dimension $d$.
D'après ce qui précède, elle possède une composante irréductible~$E$ 
telle que $\int_E P>0$. 
Soit $u\in Z'$ un point de~$q(E)$.
On a $\pi(u)=t$, donc $u\in U\cap Z'$  par choix de~$t$. 
Par platitude, $q$ est purement
de dimension relative $n-d$ au-dessus de $U$, donc $E$ est une
composante irréductible de $q^{-1}(u)$. 
Comme $c_1(L_1)\dots c_1(L_d)$ est nef,
on a $\int_{q^{-1}(u)}c_1(L_1)\dots c_1(L_d)>0$.
En vertu de la platitude de $q$ au-dessus de $U$,
le degré de $c_1(L_1)\dots c_1(L_d)$ sur la fibre générique de $q$ 
est alors également strictement positif
(de nouveau par conservation du nombre), ce qu'il fallait démontrer.
\end{proof}

\begin{prop}\label{detectgen} 
Soit~$K$
un corps gradué, soit  $V$ un
$K$-schéma gradué intègre et
propre de dimension $n$ et soit $(L_1,\dots,L_d)$ 
une famille de fibrés en droites sur~$V$.
On suppose que la classe $c_1(L_1)\dots c_1(L_d)$ est numériquement
effective 
et numériquement non triviale.
Il existe un $K$-schéma gradué intègre et propre~$W$, 
un morphisme birationnel $p\colon W\to V$, un polyrayon $r$
de longueur $n-d$ et un morphisme dominant
$q\colon W\to \P^{1,r}_K$ tels que le 
degré de $c_1(p^*L_1)\dots c_1(p^*L_d)$ sur la fibre générique
de $q$ soit strictement positif.
\end{prop}
\begin{proof}
Choisissons une sous-variété irréductible fermée~$H_0$
de~$V$ de dimension~$d$ telle que $\int_{H_0} P>0$ 
et soit $\xi$ le point générique de $H_0$. 
Le lemme \ref{lemme-suite-reguliere}
assure l'existence d'un ouvert $U$ de $V$ rencontrant $H_0$,
et de $n-d$ éléments homogènes 
$f_1,\ldots, f_d$ 
de $\mathscr O_X(U)$, de degrés respectifs $r_1,\ldots, r_{n-d}$,
tel que le lieu des zéros des $f_i|_U$ soit
égal à $H_0\cap U$. Pour tout $i$,
la famille $(f_1,\ldots,f_{n-d})$ définit un
morphisme de $U$ vers $\A^r_K\subset \P^{(1,r)}_K$ où
$r=(r_1,\ldots, r_{n-d})$ ; il existe alors
un morphisme propre et birationnel $p\colon W\to V$, qui est un isomorphisme
au-dessus de $U$, tel que $U\to  \P^{(1,r)}_K$ s'étende en un morphisme
$q\colon W\to  \P^{(1,r)}_K$. L'adhérence de $H_0\cap U$ dans $W$ est une
variété irréductible $H$ de dimension $d$ qui est par construction une
composante irréductible de $q^{-1}(0)$ ; comme $H$ domine $H_0$, on a
$\int_H P>0$. Il résulte alors du lemme~\ref{deggen} ci-dessus que
le degré de $c_1(L_1)\dots c_1(L_d)$ 
sur la fibre générique de $q$ est strictement positif.
\end{proof}

\begin{theo}
Soit $(X,x)$ un germe d'espace $k$-analytique sans bord, 
soit $(\overline L_1,\dots,\overline L_d)$ 
une famille finie de fibrés en droites sur~$X$ munis de métriques
paralinéaires.
On suppose que le courant $c_1(\overline L_1)\wedge\dots\wedge c_1(\overline L_d)$ est positif et non nul,
et qu'il existe un fermé de Zariski de~$(X,x)$ de dimension~$d_k(x)$.
Il existe alors une immersion $\phi\colon (Y,x)\hra (X,x)$
d'un germe d'espace analytique sans bord de dimension~$d$ 
tel que le courant $c_1(\phi^*\overline L_1)\dots c_1(\phi^*\overline L_d)$
soit non nul.
\end{theo}

En particulier, pour $d=1$, on a le corollaire suivant.
\begin{coro}
Soit $(X,x)$ un germe d'espace $k$-analytique sans bord, 
soit $\overline L$ un fibré en droites sur~$X$ muni d'une métrique
paralinéaire psh tel que le courant $c_1(L)$ n'est pas nul. 
S'il existe un fermé de Zariski de~$(X,x)$ de dimension~$d_k(x)$,
il existe un germe de courbe $(Y,x)$ sur~$(X,x)$
tel que le courant $c_1(L|_Y)$ ne soit pas nul.
\end{coro}

\begin{proof}
Posons $n = d_k(x)$.
Quitte à remplacer $(X,x)$ par le fermé de Zariski de dimension~$n$
dont on a postulé l'existence, on peut supposer que la dimension de~$(X,x)$
est égale à~$n$.
Posons $m=n-d$. 

D'après le théorème~\ref{theo.prod-courants-paralineaires-local},
la classe $c_1(\widetilde L_{1,x})\dots c_1(\widetilde L_{d,x})$
sur $\zr{\hrt x}{\widetilde k}$ est nef et numériquement non triviale.
Soit
$(V,M_1,\dots,M_d)$
un modèle de $(\zr{\hrt x}{\widetilde k} \widetilde L_{1,x},\dots,\widetilde  L_{d,x})$;
la classe $c_1(M_1)\dots c_1(M_d)$ sur~$V$ est donc nef et numériquement
non triviale.
En vertu du lemme~\ref{detectgen}, quitte à modifier~$V$,
on peut supposer
qu'il existe un polyrayon
$r=(r_1,\ldots, r_m)$  et un morphisme dominant
$q\colon W\to \P_K^{(1,r)}$
tel que le degré de $c_1(M_1)\dots c_1(M_d)$
sur la fibre générique de~$q$ soit strictement positif.

Choisissons un système de coordonnées affines $T_1,\ldots, T_{m}$ 
sur $\A^r_{\tilde k}\subset \P^{(1,r)}_{\widetilde k}$ (chaque $T_i$ étant homogène
de degré $r_i$). 
Pour tout $i$, on note~$f_i$ l'image de $T_i$ dans $\widetilde k(V)$
\emph{via} la flèche $q$. 
Notons~$W$ la fibre générique de~$q$ ; c'est naturellement un schéma
gradué 
intègre de dimension~$d$ sur~$\widetilde k(f_1,\dots,f_m)$, de corps
gradué des fonctions
$\widetilde k(V)=\hrt x$.
Par hypothèse$\int_W c_1(M_1|_W)\dots c_1(M_d|_W)$ est
strictement positif.

Le germe $(X,x)$ est sans bord, et \emph{a fortiori} presque bon. 
Il existe donc 
$m$ fonctions
inversibles $g_1,\ldots,g_m$ sur 
le germe $(X,x)$ telles que~$\widetilde{g_i(x)}=f_i$
pour tout $i\leq m$. 
Soit $g\colon (X,x)\to (\A^{m,\an}_k,z)$ 
le morphisme de germes induit par les $g_i$.
Démontrons que le germe $(Y,x)=(g^{-1}(\eta_r),x)$ convient.

Démontrons tout d'abord que l'on a $\dim(Y,x)=d$ et que $d_{\hr{\eta_r}z}(x)=d$.
Pour tout point $y\in Y$, on a 
\[ n\geq d_k(y)=d_{\hr {\eta_r}}(y)+d_k(\eta_r)=d_{\hr{\eta_r} }(y)+m \]
de sorte $d_{\hr {\eta_r}}(y)\leq d$. Cela démontre $\dim(Y,x)\leq d$.
Comme $d_k(x)=n$, le même raisonnement fournit l'égalité
$d_{\hr {\eta_r}}(x)=d$, d'où l'égalité $\dim (Y,x)=d$. 

Posons $\Lambda_i=L_i|_Y$. 
Compte tenu du théorème~\ref{theo.prod-courants-paralineaires-local},
il reste à démontrer
que 
$\int_{\zr{\hrt x}{\hrt z}}c_1(\widetilde \Lambda_{1,x})\dots c_1(\widetilde \Lambda_{d,x})$
est non nulle.
Or par fonctorialité de la formation du fibré résiduel
$(W,M_1|_W,\ldots, M_d|_W)$ est un modèle
de $(\zr{\hrt x}{\hrt z},\widetilde \Lambda_{1,x},\dots, \widetilde \Lambda_{d,x})$. 
L'intégrale ci-dessus est donc égale à
$\int_W c_1(M_1|_W)\dots c_1(M_d|_W)$ sur~$W$ qui est strictement positive. 
\end{proof}

\begin{rema}
Dans les énoncés de cette section, on peut remplacer
la classe $c_1(L_1)\dots c_1(L_d)$, produit de $d$ premières
classes de Chern de fibrés en droites, par un polynôme homogène de degré~$d$
en les classes de Chern de fibrés vectoriels
qui est numériquement effectif et numériquement non trivial.
Le lemme~\ref{deggen} et la proposition~\ref{detectgen}
se transposent \emph{verbatim}, avec leurs démonstrations.
Le théorème~\ref{detectgen} fait référence à une notion
de classes de Chern de fibrés vectoriels munis de
métriques  paralinéaires pour lesquelles nous n'avons pas de définition.
Comme on l'a vu dans sa preuve, la positivité analytique
était immédiatement traduite en termes de théorie de l'intersection
résiduelle, via le théorème~\ref{theo.prod-courants-paralineaires-local},
et c'est ainsi que sa généralisation peut se formuler
en l'état actuel de la théorie; modulo cette modification, la démonstration  
est identique.  On laisse les détails aux lecteurs intéressés.
\end{rema}

\begin{rema} Nous ne savons pas si l'hypothèse qu'il existe
un fermé de Zariski de dimension~$d_k(x)$ 
(c'est-à-dire que $x$ est, localement, un point d'Abhyankar de rang 
maximal d'un fermé analytique)
est nécessaire.
\end{rema}

\chapter{Variétés abéliennes}

\section{Métriques sur les fibrés en droites sur les variétés abéliennes}

\begin{lemm}\label{lemm.tate}
Soit $X$ un $k$-espace analytique propre.
Soit $L$ un fibré en droites sur~$X$.
Soit $w$ un entier~$\geq 2$, soit $f\colon X\ra X$ un morphisme
et soit $\eps\colon f^*L\simeq L^w$ un isomorphisme de fibrés en droites.
Il existe une unique métrique continue sur~$L$
qui fait de l'isomorphisme~$\eps$ une isométrie.

Si $L$ possède une métrique psh  
(resp.\ psh lisse, par exemple si $L$ est ample), 
alors cette métrique est psh (resp. globalement psh).
\end{lemm}
\begin{proof}
Soit $\norm\cdot$ une métrique continue sur~$L$.
On en déduit une métrique sur~$f^*L$ par fonctorialité,
puis une métrique sur~$L^w$ au moyen de~$\eps$,
et enfin une métrique sur~$L$. La métrique~$\norm\cdot'$ ainsi
définie est  l'unique métrique continue sur~$L$ qui fait
de l'isomorphisme~$\eps$ une isométrie si l'on munit le membre de gauche
de la métrique $f^*(L,\norm\cdot)$ et le membre de droite
de la métrique $(L,\norm\cdot')^w$.
L'ensemble $M(L)$ des métriques continues sur~$L$ est un espace
affine dont l'espace de translations est 
l'espace de Banach $\mathscr C(X,\R)$.  Cela munit~$M(L)$
d'une structure canonique d'espace métrique complet.
En outre, l'application $\norm\cdot\mapsto \norm\cdot'$ 
est $(1/w)$-lipschitzienne. Elle a donc un unique point fixe~$\norm\cdot_0$.

D'après Picard, ce point fixe est la limite des métriques
définies par l'itération de l'application $\norm\cdot\mapsto\norm\cdot'$
à partir de n'importe quelle métrique. Si nous partons
d'une métrique psh (resp. psh lisse), toutes les métriques en jeu
sont psh (resp. psh lisses) et convergent uniformément vers la
métrique~$\norm\cdot_0$, laquelle est donc psh (resp. globalement psh).
\end{proof}

\subsection{}
Soit $A$ une $k$-variété abéloïde, c'est-à-dire un espace $k$-analytique
en groupes qui est propre, connexe et lisse; par exemple, l'analytifiée
d'une $k$-variété abélienne.

Soit $L$ un fibré en droites sur~$A$.
Pour toute partie~$I$ de~$A^3$,
soit $q_I\colon A^3\to A$ l'homomorphisme donné par la somme
des coordonnées d'indices dans~$I$.
Le théorème
du cube affirme que le fibré en droites $\mathscr D_3(L)$
sur~$A^3$ défini par
\[ \mathscr D_3(L) 
      = \bigotimes_{I\subset\{1,2,3\}} (q_I)^*L^{\otimes (-1)^{\Card(I)+1}}
\]
est trivial. 
Plus précisément, ce fibré possède une trivialisation canonique~$\tau_L$;
elle satisfait à certaines relations de symétrie et de cocycle
pour lesquelles nous renvoyons à~\cite{breen1983,moret-bailly85b}
dans le cas des variétés abéliennes
et à~\cite{Lutkebohmert-2016} pour le cas des variétés abéloïdes générales.
En particulier, la formation de~$\mathscr D_3(L)$ et de sa trivialisation
canonique commute à l'image réciproque par un morphisme de variétés abéloïdes,
au produit tensoriel de fibrés en droites et au dual.
De plus, la restriction de~$\tau_L$ à $A\times A\times\{0\}$
coïncide avec la trivialisation évidente de la restriction de~$\mathscr D_3(L)$
à ce sous-espace.

On définit de manière analogue le fibré en droites~$\mathscr D_2(L)$
sur~$A^2$:
\[ \mathscr D_2(L) = q_{12}^*L \otimes q_1^*L^{-1}\otimes q_2^*L^{-1}
 \otimes q_\emptyset^* L. \]
Il est canoniquement rigidifié; ses restrictions à $A\times\{0\}$
et $\{0\}\times A$ sont canoniquement trivialisées.
Il est trivialisable si et seulement si 
$L$ est invariant par translation :
pour tout morphisme $x\colon S\to A$, on a un isomorphisme
$L \simeq t_x^*L$ sur $A\times S$. 

Toute métrique sur~$L$ induit des métriques sur~$\mathscr D_3(L)$
et $\mathscr D_2(L)$.

\begin{defi}
Nous dirons qu'une métrique continue sur~$L$ est \emph{cubiste}
si la norme de sa trivialisation canonique pour la métrique
induite sur~$\mathscr D_3(L)$ est identiquement égale à~$1$.
\end{defi}

\begin{lemm}\label{lemm.cubiste-tens}
\begin{enumerate}
\item
Soit $L$ et $M$ des fibrés en droites sur~$A$, munis de métriques
cubistes; la métrique produit tensoriel sur $L\otimes M$ est cubiste.

\item
Soit $L$ un fibré en droite sur~$A$ muni d'une métrique cubiste;
la métrique duale sur~$L^{-1}$ est cubiste.
\end{enumerate}
\end{lemm}
\begin{proof}
On un isomorphisme naturel 
$\mathscr D_3(L)\otimes \mathscr D_3(M)\simeq \mathscr D_3(L\otimes M)$ ;
c'est une isométrie.
Comme il applique le produit tensoriel des deux trivialisations
canoniques sur la trivialisation canonique, la première
assertion s'en déduit.

De même, on a une isométrie naturelle $\mathscr D_3(L^{-1})\simeq \mathscr D_3(L)^{-1}$, compatible aux trivialisations canoniques, d'où
la seconde assertion.
\end{proof}

\subsection{}
Soit $L$ un fibré en droites sur~$A$; fixons-en une rigidification~$\eps$,
c'est-à-dire une trivialisation en l'origine de~$A$.

Supposons que $[-1]^*L$ soit isomorphe à~$L^{-1}$ 
et fixons un tel isomorphisme compatible aux rigidifications
induites par~$\eps$.
On déduit du théorème du cube par récurrence,
par applications successives de l'isomorphisme~$\tau_L$,
qu'il existe pour tout entier~$n\in\Z$ un unique isomorphisme
$[n]^*L\simeq L^{n}$  qui est compatible aux rigidifications
sur~$[n]^*L$ et~$L^n$ induites par~$\eps$. 
La construction de cet isomorphisme entraîne que c'est une isométrie
si $L$ est muni d'une métrique cubiste telle que $\norm\eps=1$.

Supposons que $[-1]^*L$ soit isomorphe à~$L$ et fixons 
un tel isomorphisme compatible aux rigidifications induites par~$\eps$.
De même, il existe alors pour tout entier~$n\in\Z$ un unique isomorphisme
$[n]^*L\simeq L^{n^2}$  qui est compatible aux rigidifications
sur~$[n]^*L$ et~$L^{n^2}$ induites par~$\eps$. 
Si $L$ est muni d'une métrique cubiste telle que $\norm\eps=1$,
cet isomorphisme est une isométrie.

\begin{lemm}
Soit $L$ un fibré en droites sur une variété abéloïde~$A$.

\begin{enumerate}
\item
Le fibré en droites~$L$ possède une métrique cubiste~$\norm\cdot$.

\item
Une métrique sur~$L$ est cubiste si et seulement si elle est de la forme
$c\norm\cdot$, pour un nombre réel strictement positif~$c$.

\item 
Soit $n$ un entier~$\geq 2$.
Supposons que $L$ est symétrique (resp. antisymétrique) et fixons 
un isomorphisme $[n]^*L\simeq L^{n^w}$, où $w=2$ (resp. $w=1$).
Une métrique sur~$L$ est cubiste si et seulement cet isomorphisme
est de norme constante.
\end{enumerate}
\end{lemm}
\begin{proof}
Comme $L\mapsto \mathscr D_3(L)$  est homogène de degré~$0$,
un multiple réel strictement positif d'une métrique cubiste sur~$L$
est encore cubiste.
Il suffit donc de démontrer que pour toute rigidification~$\eps$ de~$L$,
il existe une unique métrique cubiste sur~$L$
telle que $\norm\eps=1$.

On commence par traiter le cas 
où $L$ est ou bien antisymétrique ($[-1]^*L\simeq L^{-1}$)
ou bien symétrique ($[-1]^*L\simeq L$);
posons $w=1$ dans le premier cas et $w=2$ dans le second.
Il existe alors pour tout~$n\geq 2$ un unique isomorphisme 
$[n]^*L\simeq L^{n^w}$ qui est compatible avec les rigidifications
induites par~$\eps$.
D'après le lemme~\ref{lemm.tate}, il existe pour tout~$n\geq 2$
une unique métrique continue sur~$L$ qui rend cet isomorphisme
une isométrie. Ces métriques sont deux à deux égales
et sont nécessairement égales à la métrique cherchée,
si elle existe.

Prouvons que ces métriques sont cubistes.
Or, l'isomorphisme $[n]^*L\simeq L^{n^w}$ induit
un isomorphisme $[n]^*\mathscr D_3(L)\simeq \mathscr D_3(L)^{n^w}$.
Cet isomorphisme est compatible aux rigidifications
données par le théorème du cube.
En outre, c'est une isométrie lorsque~$L$ est muni de la métrique
construite ci-dessus. Par suite, la métrique construite sur~$L$
induit la métrique triviale sur le fibré en droites
trivialisé~$\mathscr D_3(L)$.

Cela prouve l'existence d'une métrique cubiste lorsque $L$ est symétrique
ou antisymétrique. Dans le cas général, la formule
$L^2\simeq (L\otimes [-1]^*L)\otimes (L\otimes [-1]^*L^{-1})$
exprime $L^2$ comme produit tensoriel d'un fibré en droites symétrique
et d'un fibré en droites antisymétrique, rigidifiés, de sorte que $L^2$
possède une métrique cubiste telle que $\norm{\eps^2}=1$.
 Sa « racine carrée » est une métrique
cubiste sur~$L$ telle que $\norm\eps=1$.

Démontrons maintenant l'unicité.
Considérons deux métriques cubistes~$\norm\cdot$ et $\norm\cdot'$
sur~$L$  telles que $\norm\eps=\norm\eps'=1$. 
Leur quotient est une métrique cubiste sur~$\mathscr O_{A}$
pour laquelle $\norm{1}=1$;
il suffit de prouver que cette métrique quotient est la métrique triviale.
On est ainsi ramené au cas $L=\mathscr O_A$ et $\eps=1$.

Posons $M=L\otimes [-1]^*L$.
Pour tout entier $n\geq 2$,
l'isométrie du cube de~$L$ fournit une isométrie
$[n]^*M\simeq M^{n^2}$ compatible avec la rigidification~$\eps$.
D'après le lemme~\ref{lemm.tate}, la métrique de~$M$ est triviale.

Posons maintenant $M'=L\otimes [-1]^*L^{-1}$.
De même, l'isométrie du cube de~$L$
fournit une isométrie $[n]^*M' \simeq (M')^n$, 
compatible avec la rigidification~$\eps$, de sorte
que la métrique de~$M'$ est elle-aussi triviale.

Comme $L^2=M\otimes M'$, la métrique de~$L$ est triviale, 
ce qu'il fallait démontrer.

La dernière assertion résulte de la construction des métriques cubistes.
\end{proof}

\begin{lemm}
Si $L$ est invariant par translation, alors la norme de toute trivialisation
de~$\mathscr D_2(L)$
est constante lorsque ce fibré en droites est muni de sa métrique
induite par une métrique cubiste sur~$L$.
\end{lemm}
\begin{proof}
Posons $M=\mathscr D_2(L)$.
Pour tout espace analytique~$S$ sur un extension valuée complète de~$k$
et tout triplet $(x,y,z)$ de $k$-morphismes de~$S$ dans~$A$,
la trivialisation~$(x,y,z)^*\tau_L$ de~$(x,y,z)^*\mathscr D_3(L)$
fournit des isomorphismes
\[ (x,y+z)^*M\simeq  (x,y)^*M\otimes (x,z)^*M   \]
et
\[ (x+y,z)^*M\simeq (x,z)^*M \otimes (y,z)^* M \]
définissant la structure de biextension sur~$M$.
Pour qu'une métrique sur~$L$ soit cubiste, il faut et il suffit
que l'un ou l'autre de ces isomorphismes soit une isométrie, lorsque
$M$ est muni de la métrique induite par celle de~$L$.

Supposons $M$ trivialisable et soit $s$ l'unique section inversible
de~$M$ compatible avec sa rigidification canonique.
Posons $f=\log(\norm s)$; c'est une  fonction continue sur~$A\times A$,
elle est en particulier bornée.
Par sa définition, le fibré~$M$ est canoniquement trivialisé
sur~$A\times\{0\}$ et $\{0\}\times A$; la section~$s$
s'y identifie alors à la section~$1$.

Considérons le cas particulier 
où $S$ est l'espace~$A^3$ et $x,y,z$ sont égales aux trois projections
de~$A^3$ dans~$A$.
Alors $  (x,y+z)^* s\otimes ((x,y)^* s \otimes (x,z) ^*s)^{-1}$
est une section du fibré~$\mathscr D_3(L)$
qui se restreint en sa trivialisation évidente
sur~$\{0\}\times A^2$; c'est donc la section~$\tau_L$.

Relâchons nos hypothèses sur $S,x,y,z$. 
Par définition d'une métrique cubiste et par ce qui précède, 
on a 
 $f\circ (x,y+z)=f\circ (x,y)+f\circ (x,z) $.
On en déduit que $f\circ (x,ny)=n f\circ(x,y)$ pour tout~$n$.
Il en résulte $\abs{f\circ (x,y)}\leq \frac1n \sup (\abs f)$
pour tout~$n$, d'où $f\circ (x,y)\equiv 0$.
Comme ceci est vrai pour tout $(S,x,y)$, il en résulte que $f$ est
identiquement nulle.
\end{proof}

\begin{prop}
Supposons que $A$ se prolonge en un schéma abélien formel~$\mathfrak A$ 
sur~$k^\circ$.
Soit $L$ un fibré en droites sur~$A$.
Il existe un fibré en droites~$\mathfrak L$ sur~$\mathfrak A$
qui prolonge~$L$ et une trivialisation
de~$\mathscr D_3(\mathfrak L)$ qui prolonge celle de~$\mathscr D_3(L)$.
La métrique formelle sur~$L$ associée à ce modèle est cubiste.
\end{prop}

\begin{proof}
L'existence d'un tel fibré en droites~$\mathfrak L$
découle du lemme~6.2.4 de~\cite{Lutkebohmert-2016}.
(Dans le cas algébrique, voir aussi le corollaire~\ref{coro.extension-pic}.)
Le fibré en droites $\mathscr D_3(\mathfrak L)$ sur~$\mathscr A^3$
est alors un modèle du fibré en droites~$\mathscr D_3(L)$ sur~$A^3$.
Comme ce dernier est trivialisable, le corollaire affirme
également que $\mathscr D_3(\mathfrak L)$ est trivialisable.
La restriction de~$\mathscr D_3(\mathfrak L)$ à~$\mathfrak A^2\times\{0\}$
est munie d'une trivialisation évidente qui se prolonge en une
unique trivialisation~$\tau_{\mathfrak L}$ de~$\mathscr D_3(\mathfrak L)$.
Sa restriction à~$A^3$ est une 
trivialisation de~$\mathscr D_3(L)$ qui induit
la trivialisation évidente sur $A^2\times\{0\}$; c'est donc
sa trivialisation canonique.
Cela entraîne que $\tau_L$  est de norme~$1$, de sorte que
la métrique formelle sur~$L$ associé au modèle $(\mathfrak A,\mathfrak L)$
est cubiste.
\end{proof}

\begin{rema}
Plaçons-nous encore sous les hypothèses de la proposition.
Notons $\overline L$ le fibré en droites~$L$ muni de la métrique
formelle déduite du prolongement~$\mathfrak L$. 
Soit $g=\dim(A)$ et soit $\gamma\in A$ l'unique point 
dont la réduction est le point générique de~$\mathfrak A_s$.

a) La mesure $c_1(\overline L)^g$ est supportée par~$\gamma$.
Lorsque $k$ est algébriquement clos, 
le théorème~\ref{theo.formule-ma-berk} assure que
sa masse est $(c_1(\mathfrak L_s)^g \cap [\mathfrak A_s])$;
c'est encore le cas dans le cas général, comme
on le voit en appliquant le résultat à l'extension
du schéma formel~$\mathfrak A$ à~$\overline k^\circ$.

Lorsque $A$ est algébrique, 
cette masse est également égale à~$(c_1(L)^g\cap [A])$
(théorème~\ref{theo.MA-lisse-masse}).

b) Il résulte de l'exemple~\ref{exem.courants-produits-formel}, \emph b), 
que les trois propriétés de positivité suivantes sont équivalentes:
\begin{enumerate}\def\theenumi{\roman{enumi}}\def\labelenumi{(\theenumi)}
\item Le courant $c_1(\overline L)$ est positif ;
\item Le courant $c_1(\overline L)$ est positif au voisinage de~$\gamma$;
\item Le fibré en droites $\mathfrak L_s$ est numériquement effectif. 
\end{enumerate}
L'implication (i)$\Rightarrow$(ii) est évidente;
l'implication (ii)$\Rightarrow$(iii) est expliquée dans l'exemple
indiqué, de même que l'implication (iii)$\Rightarrow$(i).
\end{rema}

Dans le paragraphe suivant, nous allons exploiter l'uniformisation
des variétés abéloïdes pour obtenir une description
similaire des métriques cubistes sur les fibrés en droites
et des puissances de leurs courants de courbure.

\section{Métriques cubistes et uniformisation}

\subsection{}
Rappelons l'uniformisation des variétés abéloïdes,
d'après~\cite{berkovich1990,bosch-l91,gubler2010,Lutkebohmert-2016};
nous renvoyons aussi le lecteur à l'appendice~\ref{app-varab}.

Soit $A$ une variété abéloïde sur~$k$; notons~$g$ sa dimension.
Il existe une variété abéloïde~$B$ à bonne réduction potentielle,
un tore algébrique~$T$ de dimension~$r$,
une extension~$E$ de~$B$ par~$T$,
une suite exacte de groupes  analytiques 
\[ 0 \to \Lambda^\an \to E \xrightarrow p A \to 0 \]
où $\Lambda$ est un $k$-groupe localement isomorphe à~$\Z^r$
pour la topologie étale de~$\Spec(k)$.

Nous dirons que la situation est déployée si $B$ a bonne réduction
et si le tore~$T$ est déployé. Dans ce cas, $\Lambda$ est isomorphe à~$\Z^r$,
la suite exacte précédente s'interprète comme le revêtement 
universel de l'espace topologique~$A^\an$ et $\Lambda$ comme son
groupe fondamental.

Nous nous plaçons désormais dans ce cas.

\subsection{}
Notons $q\colon E\to B$ la projection de l'extension~$E$.
Son noyau est le tore déployé~$T$,
de sorte que cette extension
est localement triviale pour la topologie de Berkovich.

L'extension~$E$ est construite à partir 
d'une extension analytique~$N$ de~$B$
par le tore affinoïde~$T^1$ ; 
cet espace~$N$ est un domaine analytique compact
de~$A$, muni d'une structure de sous-groupe formellement
analytique compact.

Soit $\mathfrak V$ un recouvrement formellement affinoïde
de~$B$ trivialisant l'extension~$N$;
pour tout $V\in \mathfrak V$, soit $f_V\colon V\to N$ une section
de la projection $q|_N\colon N\to B^\an$.
Sur $V\cap W$, les sections~$f_V$ et~$f_W$ diffèrent de la multiplication
par un morphisme $f_{VW}\colon V\cap W \to T^1$.
Ces morphismes $(f_{VW})$ forment un cocycle à valeurs dans~$T^1$
qui représente l'extension~$N$, mais également, 
si on les considère comme un cocycle à valeurs dans~$T$,
l'extension~$E$.

Pour $V\in\mathfrak V$, la section~$f_V$ permet d'identifier $q^{-1}(V)$
au produit $V\times T$; composons la projection sur le second facteur
avec l'application de tropicalisation de~$T$ vers~$T_\trop$;
on obtient une application~$\phi_V$ de $q^{-1}(V)$ dans~$T_\trop$
qui est lisse et~\pl:
lorsqu'on identifie~$T$ à~$\gm^r$
et $T_\trop$ à~$\R^r$, chaque coordonnée est 
le logarithme de la valeur absolue d'une fonction holomorphe inversible
(un caractère du tore).

Comme le cocycle $(f_{VW})$ est à valeurs dans~$T^1$,
ces applications $\phi_V$ se recollent en une application
G-lisse, \pl, de~$E$ dans~$T_\trop$. Notons-la~$\tau$.

La restriction de~$\tau$ au tore~$T$ est l'application de tropicalisation
usuelle de~$T$ sur~$T_\trop$. Comme $\tau|_{T}$ est propre
et que $B$ est compacte, l'application $\tau\colon E\to T_\trop$
est surjective et propre.

Si $E=B$, le tore~$T$ est de dimension~$0$ 
et $\tau$ est l'application constante.

L'image du réseau~$\Lambda$ par~$\tau$
est un réseau cocompact de~$T_\trop$ 
(théorème~1.2 de \cite{BoschLutkebohmert-1991}).
Par suite, $\tau$ induit une application surjective, \pl,
de~$A$ dans le tore réel $T_\trop/\tau(\Lambda)$.

Notons $\beta$ le point de Gauss de~$B$.
Dans \cite[\S6.5]{berkovich1990}, 
Berkovich construit une rétraction par déformation
de~$E$ sur le squelette du tore $T_{\mathscr H(\beta)}$
qui s'identifie à~$T_\trop$ via~$\tau$; nous
le noterons~$S(E)$. L'action de~$\Lambda$
par translation fixe le point~$\beta$ et stabilise ce squelette.
Par passage au quotient, cela induit une rétraction par déformation
de~$A$ sur le sous-espace $S(A)=S(E)/\Lambda$
qu'il appelle le squelette de~$A$;
par l'application~$\tau$, ce squelette s'identifie
au tore réel $T_\trop/\tau(\Lambda)$.
Pour toute extension valuée~$K$ de~$k$,
la rétraction de~$A$ sur~$S(A)$ induit un morphisme
de groupes de~$A(K)$ sur~$S(A)$.

\begin{rema}
Supposons que $T$ soit le tore~$\gm$.
La variété abéloïde~$B$ a bonne réduction, prolongeons-la
en un schéma formel abélien~$\mathfrak B$ sur $k^\circ$.
L'extension formelle~$N$ de~$B$ par~$\gm^1$ correspond
à un fibré en droites~$\mathfrak Q$, algébriquement équivalent à zéro,
sur le schéma formel abélien~$\mathfrak B$.
Notons~$Q$ le fibré en droites sur~$B$ qui correspond à~$\mathfrak Q$ 
par passage à la fibre générique.
L'extension~$E$ est un ouvert de l'espace total du fibré~$Q$ et 
la métrique formelle de~$Q$ déduite de~$\mathfrak Q$,
considérée comme fonction sur l'espace total de ce fibré
est la fonction~$\tau$.

Le fibré résiduel~$\mathfrak Q_s$ sur~$\mathscr B_s$
est numériquement trivial,
car algébriquement équivalent à zéro.
D'après le théorème~\ref{theo.prod-courants-paralineaires-local},
le courant $c_1(\overline Q)$ est donc nul.
Cela entraîne que le courant $\ddc\tau$ est nul:
\emph{la fonction~$\tau$ est harmonique.}

Pour que la fonction~$\tau$ soit lisse, il est nécessaire
et suffisant que la métrique formelle induite par~$\mathfrak Q$
soit lisse.  D'après la proposition~\ref{prop.metrique-lisse-pl},
cela signifie que $\mathfrak Q$ est de torsion 
sur toute sous-variété intègre de~$\mathscr B_s$.
Ainsi, $\tau$ est lisse si et seulement $\mathfrak Q_s$
est de torsion.

Si le corps résiduel de~$k$ est algébrique sur  un corps fini,
cette dernière propriété est automatique, mais ce n'est pas le cas en général.

Lorsque $T$ est un tore de dimension arbitraire, on peut 
l'identifier à~$\gm^r$, de sorte que $\tau$ s'identifie
à une famille $(\tau_1,\dots,\tau_r)$ de fonctions numériques sur~$E$.
Ces fonctions sont G-lisses,  harmoniques ; 
pour qu'elles soient toutes lisses, il faut et il suffit
l'extension résiduelle de~$\mathscr B_s$
par~$T_{\tilde k}$ soit isotriviale.
\end{rema}

\subsection{}
Soit $L$ un fibré en droites sur~$A$.
Choisissons-en une trivialisation~$\ell$ en l'origine de~$A$
et munissons alors $L$ de l'unique métrique cubiste pour laquelle $\norm\ell=1$.

Le fibré en droites $p^*L$ sur~$E$ provient de~$B$:
il existe un fibré en droites~$M$ sur~$B$ 
et un isomorphisme $\eps\colon p^*L \simeq q^*M$.
Au-dessus de l'élément neutre de~$E$, le fibré~$p^*L$
possède la trivialisation~$p^*\ell$,
et le fibré~$q^*M$
possède possède une trivialisation déduite de la structure formelle
de~$M$, unique à unité près;
on normalise~$\eps$ de sorte qu'il applique l'une sur l'autre.

Le choix d'un tel isomorphisme~$\eps$ munit~$q^*M$ d'une donnée de
descente de~$E$ à~$A$, c'est-à-dire ici
d'une famille d'isomorphismes $c=(c_\lambda)_{\lambda\in\Lambda}$,
où, pour $\lambda\in \Lambda$,
$c_\lambda$ est un isomorphisme $q^*M \to t_\lambda^* q^*M$
de fibrés en droites sur~$E$ 
($t_\lambda$ désigne la translation par~$\lambda$
dans~$E$), satisfaisant en outre à la condition de cocycle
\[ c_{\lambda+\lambda'} =  t_{\lambda'}^* c_\lambda \circ c_{\lambda'}. \]

Pour aller plus loin, nous reprenons les notations de l'appendice
(A.7.15-17). En particulier, on note~$q_L$
la forme quadratique sur~$T_\trop$ associée au fibré en droites~$L$.

On veut comparer la métrique de $q^*M$ 
déduite de la métrique formelle de~$M$
et la métrique cubiste de $p^*L$ au travers 
de l'isomorphisme qui identifie ces fibrés.

Remarquons que l'application de tropicalisation~$\tau$
s'exprime en termes de ces métriques formelles
par la formule 
\[ \langle \tau(t), \lambda'  \rangle = \log \norm{ s_{\lambda'} } (t) ,\]
pour $t\in E$.

\begin{prop}\label{prop.metrique-formelle-cubiste}
Fixons $(\eps,M,c)$ comme ci-dessus. 
Il existe
une unique fonction~$f$ sur~$T_\trop$, à valeurs réelles
telle  que l'on ait 
\[ \log\norm{s}^{-1} - \log \norm{\eps(s)}^{-1} = f \circ\tau \]
pour toute section inversible locale~$s$
de~$p^*L$. Cette fonction est quadratique 
et sa composante homogène de degré~$2$ 
est égale à~$\frac12 q_L$.
\end{prop}
\begin{proof}
Pour tout $\lambda\in\Lambda$, nous allons étudier
comment l'isomorphisme $c_\lambda$ modifie les métriques 
sur~$q^*M$ et $t_\lambda^*q^*M$ déduites de la métrique formelle de~$M$.
Au-dessus d'un point~$t$ de~$E$,
on a 
\[ \log\norm{c_\lambda} = \log \norm{r(\lambda)\otimes r(0)^{-1}}
+ \log \norm{s_{u(\lambda)}}(t)
= \log \norm{r(\lambda)\otimes r(0)^{-1}}  + \langle \tau(t), 
u(\lambda)\rangle \]

On déduit de la formule
\[ r(\lambda+\mu)\otimes r(\lambda)^{-1}\otimes r(\mu)^{-1}\otimes r(0)
 = \langle \lambda, u(\mu)\rangle \]
que  la fonction
$ \lambda \mapsto \log\norm{r(\lambda)}$
est quadratique sur~$\Lambda$, 
et que sa composante homogène de degré~2 est donnée
par $\lambda \mapsto \frac12 \log \norm{\langle \lambda, u(\lambda)}$.
Il existe une unique fonction quadratique~$f$ sur~$T_\trop$
telle que 
\[ \log\norm{r(\lambda)} = f\circ \tau(\lambda) \]
pour tout $\lambda\in \Lambda$.

Munissons $q^*M$ de la métrique~$\norm\cdot_f$ obtenue
en multipliant la métrique formelle par la fonction $e^{-f\circ\tau}$.
Pour cette nouvelle métrique, la donnée de descente a pour norme
\begin{align*}
 \log\norm{c_\lambda}_f 
& = \log \norm{r(\lambda)}- \log \norm{r(0)}
+ \langle \tau(t), u(\lambda)\rangle
+ f(\tau(t))- f(\tau(t+\lambda))  \\
& = -f\circ \tau (t+\lambda)+ f\circ \tau(t) + f \circ \tau(\lambda) - f\circ\tau(0) 
+ \langle \tau(t), u(\lambda) \rangle. 
\end{align*}
C'est une expression linéaire en~$t$ et en~$\lambda$.
Pour vérifier qu'elle est nulle, il suffit de prouver qu'elle
s'annule lorsque $t\in\Lambda$.
Soit donc $\mu\in\Lambda$. Par construction de~$f$, on a
\[  f\circ \tau (\mu+\lambda)- f\circ \tau(\mu) - f \circ \tau(\lambda) + f\circ\tau(0) = \log \norm{\langle\lambda, u(\mu)\rangle}  
= \log \norm {s_{u(\lambda)}(\mu)} 
= \langle \tau(\mu), u(\lambda)\rangle 
\]
d'où l'annulation requise.

La métrique formelle sur~$M$ est cubiste, donc celle
sur~$q^*M$ l'est également. Comme $f\circ \tau$ est une fonction
quadratique sur~$E$, la métrique modifiée $\norm\cdot_f$ 
est encore cubiste.

Compte tenu de l'isomorphisme~$\eps$,
la métrique modifiée~$\norm\cdot_f$ descend en une métrique sur~$L$,
cubiste. C'est donc la métrique cubiste canonique de~$L$, à un scalaire
près.  (Ce scalaire vaut~$1$ si $M$ et~$L$ sont rigidifiés en l'origine
et si $\eps$ et~$r$ sont supposés respecter ces rigidifications.)
Cela démontre qu'il existe une fonction quadratique~$f$ sur~$T_\trop$
qui vérifie la première partie de l'énoncé.
L'unicité d'une telle fonction résulte de la surjectivité de~$\tau$.
Enfin, on a vu que la partie homogène de degré~$2$ de~$f$
est égale à $\lambda \mapsto \frac12 \norm{\langle \lambda ,u(\lambda)\rangle}$
sur~$\Lambda$; elle y coïncide donc avec~$\frac12 q_L\circ \tau$,
ce qui conclut la démonstration.
\end{proof}

\begin{coro}
\begin{enumerate}
\item
La métrique cubiste de~$\bar L$ est d-psh.
\item
On a $c_1(p^*\bar L) = c_1(q^*\overline M) + \tau^*(\di\dc  \frac12 q_L)$.

\item
Le courant de courbure $c_1(\bar L)$ est invariant par translation sur~$A$;
il est nul si $L$ est invariant par translation.
\end{enumerate}
\end{coro}
\begin{proof}
\begin{enumerate}
\item
La relation
entre la métrique cubiste de~$p^*L$ et la métrique formelle de~$q^*M$
donnée par la proposition~\ref{prop.metrique-formelle-cubiste} prouve
que la métrique de~$p^*\overline L$ est d-psh: c'est la somme
d'une métrique formelle et d'une métrique G-lisse.
\item
De cette relation découle
l'égalité de courants
\[ c_1(p^*\bar L) = c_1(q^*\bar M) + \ddc [ f\circ\tau]. \]
Comme $\tau$ est harmonique, on  a 
\[ \ddc [f \circ \tau] = \ddc[\tau^* f] = \tau^* [\ddc(f)], \]
en vertu du corollaire~\ref{coro.harm-diff}.
Comme $f$ est quadratique et que sa composante homogène
de degré~$2$ est égale à~$\frac12 q_L$, on a 
aussi 
$ \ddc f = \frac12 \ddc q_L $ sur $T_\trop$, si bien que
\[ c_1(p^*\bar L) = c_1(q^*\bar M) + \tau^* \ddc(\frac12 q_L). \]

\item
Notons $t_a\colon A \to A$ la translation par un point $a\in A(k)$.
Nous devons démontrer que $c_1(t_a^*\overline L)=c_1(\overline L)$.
Le fibré en droites sous-jacent à $t_a^*\overline L\otimes \overline L^{-1}$
est $N=t_a^*L\otimes L^{-1}$ et sa métrique est cubiste.
On a donc $c_1(t_a^*\overline L)=c_1(\overline L)+c_1(\overline N)$.
Comme $N$ est algébriquement équivalent à~$0$,
il suffit de démontrer que l'on a $c_1(\overline L)=0$
si $L$ est algébriquement équivalent à~$0$.

Dans ce cas, la proposition~\ref{prop.unif-pic0}
entraîne que $q_L=0$, d'où $\ddc(q_L)=0$.
Par la même proposition, le fibré formel~$\mathfrak M$
sur~$\mathfrak B$ est invariant par translation, si bien que
le fibré en droites~$\mathfrak M_s$ 
sur la variété abélienne~$\mathfrak B_s$
est algébriquement équivalent à~$0$, et donc numériquement
trivial.
D'après le corollaire~\ref{coro.PL-psh-c1pos} (appliqué
à~$\overline M$ et son dual), on a $c_1(\overline M)=0$.

Cela prouve que $c_1(\overline L)=0$.
\qedhere
\end{enumerate}
\end{proof}

\begin{rema}\label{rema.omegaL}
Comme la forme $\ddc(\frac12q_L)$ est invariante par translations
sur~$T_\trop$ et comme $\tau$ est équivariant pour l'action
de~$\Lambda$, la forme $\tau^*(\ddc \frac12q_L)$ 
est invariante par translations par~$\Lambda$,
de sorte qu'il existe une unique G-forme~$\omega_L$ sur~$A$
telle que $p^*\omega_L=\tau^*(\ddc\frac12q_L)$.
\end{rema}

\begin{coro}
\begin{enumerate}
\item
La mesure $c_1(p^*\bar L)^g$ est une mesure de Haar sur~$S(E)$.

\item
La mesure $c_1(\bar L)^g$ est l'unique mesure de Haar sur~$S(A)$
de masse totale~$\deg_L(A)$.
\end{enumerate}
\end{coro}
\begin{proof}
\begin{enumerate}
\item
La formule précédente écrit
le courant $c_1(p^*\bar L)$ comme la somme du courant paralinéaire
$c_1(q^*\bar M)$ et d'un courant associé à une G-forme G-lisse
de type~$(1,1)$ localement harmoniquement tropicalisable,
tous déduits par~$\ddc$ de métriques d-psh.
 
Comme le produit de courants à la Bedford--Taylor est multilinéaire, il
en résulte l'égalité
\[ c_1(p^*\bar L)^g = \big(c_1(q^*\bar M) + \tau^* \ddc(\frac12 q_L)\big)^g
= \sum_{d=0}^g \binom gd c_1(q^*\bar M)^{g-d} \wedge
(\tau^*\ddc(\frac12 q_L))^d . \]
Dans cette formule, les produits sont pris au sens de la théorie
de Bedford--Taylor.
Le corollaire~\ref{coro.produit-BT=produit-G}
affirme qu'elle reste valable lorsque,
pour tout~$d$, on interprète 
le terme d'indice~$d$ du second membre 
comme le produit du courant paralinéaire fermé 
(différence de courants de Bedford--Taylor)
$c_1(q^*\bar M)^{g-d}$ et de la G-forme 
(localement harmoniquement tropicalisable)
$\tau^*(\ddc\frac12 q_L)^d$.

Le terme d'indice~$d$ du membre de droite est nul dès que $g-d>r$,
la dimension de l'image de~$\tau$. Posons $b=\dim(B)$;
ce terme est nul si $d>b$ car
$q^*\bar M$ est de dimension de définition~$\leq b$ en tout point.
Ainsi, seul le terme d'indice $b$ persiste et l'on obtient
\[ c_1(p^*\bar L)^g 
=  \binom gr c_1(q^*\bar M)^{b} \wedge (\tau^*\ddc(\frac12 q_L))^r . \]

Soit $u$ une fonction lisse à support compact sur~$E$. On a donc
\begin{align*}
 \langle c_1(p^*\bar L)^g , u\rangle
& =  \binom gr \langle c_1(q^*\bar M)^{b} \wedge (\tau^*\ddc(\frac12 q_L))^r ,u\rangle  \\
& =  \binom gr \langle c_1(q^*\bar M)^{b} , u (\tau^*\ddc(\frac12 q_L))^r \rangle .
\end{align*}

D'après la proposition~\ref{prop.MA-relatif},
le courant $ c_1(q^*\bar M)^b$ 
est un multiple $c\delta_{q^{-1}(\beta)}$ 
du courant d'intégration sur la fibre~$q^-1(\beta)$.
En effet, comme la fibre spéciale
de~$B$ est irréductible et de dimension~$b$,
seul son point générique~$\beta$ peut intervenir.
Ainsi,
\[ \langle c_1(p^*\bar L)^g , u\rangle
= c \binom gr \int_{q^{-1}(\beta)} u (\tau^*\ddc(\frac12 q_L))^r 
= c\binom gr \int_{q^{-1}(\beta)} u \tau^*\omega, \]
où l'on a posé $\omega=\ddc(\frac12 q_L)^r$. Comme $q_L$ est quadratique,
$\omega$ est une $(r,r)$-forme
invariante par translations sur~$T_\trop$.
Ainsi, $\Sigma_\tau=S(E)$ est  un support fort de la
G-forme $u\tau^*\omega$, de sorte que
\[ \langle c_1(p^*\bar L)^g , u\rangle
= c \binom gr \int_{S(E)} \langle u \tau^*\omega, \mu_E\rangle, \]
où $\mu_E$ est le calibrage canonique de~$S(E)$.

On peut identifier la fibre~$q^{-1}(\beta)$ avec $T_{\hr \beta}$
de sorte que la fonction~$\tau$ s'identifie à l'application
de tropicalisation de~$T_{\hr \beta}$ sur~$T_\trop$,
et $S(E)$ s'identifie à~$T_\trop$ muni de son calibrage canonique.
La contraction $\langle \tau^*\omega,\mu_E\rangle$
s'identifie à une $r$-forme invariante par translations~$\alpha$ sur~$S(E)$
et 
\[ \langle c_1(p^*\bar L)^g , u\rangle
= c \binom gr \int_{S(E)} u(x) \alpha(x) .\]
Ainsi, $c_1(p^*\bar L)^g $ est une mesure de Haar sur~$S(E)$, comme annoncé.

\item
Comme $p$ est une application de revêtement,
on peut tirer en arrière tout courant~$T$ sur~$A$ et en déduire
un courant~$p^*T$ sur~$E$. La construction de $c_1(\bar L)^g$
entraîne que $p^*c_1(\bar L)^g= c_1(p^*\bar L)^g$.
D'après \emph b), ce courant $p^*c_1(\bar L)^g$ est une 
mesure de Haar~$\nu$ sur~$S(E)$. 
L'application~$p$ fait de~$S(E)$ un revêtement de~$S(A)$
(c'est même son revêtement universel);
considérons la mesure de Haar~$\mu$ sur~$S(A)$ 
telle que $p^*\mu =\nu$. Les courants $p^*c_1(\bar L)^g$
et $p^*[\mu]$ sont égaux; il en résulte que les courants
$c_1(\bar L)^g$ et $[\mu]$ sont égaux.
D'autre part, on sait que la masse totale du courant $c_1(\bar L)^g$
est égale à $\deg_L(A)$ (théorème~\ref{theo.MA-lisse-masse}).
Le lemme en résulte.
\qedhere
\end{enumerate}
\end{proof}
\begin{coro}\label{coro.structure-mesure}
Soit $U$ un ouvert de~$A$,
soit $X$ un sous-espace analytique irréductible fermé de~$U$, 
soit~$d$ la dimension de~$X$.
Pour tout $j\in\N$,
le courant $c_1(q^*\overline M|_{p^{-1}(X)})^j$ sur~$p^{-1}(X)$
descend  en un courant paralinéaire~$S_j$ sur~$X$, de bidimension~$(d-j,d-j)$.

Soit $\omega_L$ l'unique G-forme sur~$A$ telle que $p^*(\omega_L)
=\tau^*(\ddc \frac12q_L)$ (remarque~\ref{rema.omegaL}).

\begin{enumerate}
\item
Pour tout entier~$e\in\{0,\dots,d\}$,
on a
\[ c_1(\overline L|_X)^e = \sum_{j=0}^e  \binom ej S_{e-j} \wedge \omega_L^j. \]
De plus, on a $S_j=0$ si $j>b$ ou si $j>d$,
et $\omega_L^j=0$ si $j>r$, si bien que la somme précédente
peut être restreinte aux entiers~$j$ tels que 
$e-j \leq \inf(b,d)$
et $j\leq r$, c'est-à-dire 
$\sup(e-b,0)\leq j\leq \inf(r,e)$.

\item
Il existe une suite $(P_0,\dots,P_d)$ de sous-espaces paralinéaires
calibrés de~$X$, où $P_j$ est de dimension~$j$,  telle que
l'on ait
\[ \langle c_1(\overline L|_X)^d , f\rangle 
= \sum_{j=0}^d \int_{P_j} f \omega_L|_{P_j}^j \]
pour toute fonction~$f$ mesurable localement bornée sur~$X$.
\end{enumerate}
\end{coro}
\begin{proof}
Rappelons les notations $b=\dim(B)$ et $r=\dim(T)$, de sorte que $g=b+r$.
Posons $Y=p^{-1}(X)$ et 
considérons l'égalité de courants
\[ c_1(p^*\overline L|_X) = c_1(q^*\overline M|_Y) 
	+ \ddc(\tau^*(\frac12q_L)|_Y) \]
sur~$Y$.
Prenons-en la puissance $e$-ième au sens de Bedford--Taylor;
on obtient
\[ c_1(p^*\overline L|_X)^e = \sum_{j=0}^d \binom ej 
c_1(q^*\overline M|_Y)^{e-j} \ddc(\tau^*(\frac12q_L)|_Y) ^j \]
Pour tout entier~$i$,
la puissance $c_1(q^*\overline M|_Y)^{i}$ est un courant
paralinéaire sur~$Y$, de bidimension~$(d-i,d-i)$;
il est différence de courants de Bedford--Taylor
et est nul si $i>b$ car sa dimension de définition est~$\leq b$.
La (1,1)-forme $\widetilde\omega_L=\tau^*(\ddc\frac12q_L)$ sur~$E$
est localement harmoniquement tropicalisable;
il en est donc de même de sa restriction à~$Y$.
Le corollaire~\ref{coro.produit-BT=produit-G}
entraîne donc que
le produit au sens de Bedford--Taylor  de $c_1(q^*\overline M|_Y)^{e-j}$
par $\widetilde\omega_L|_Y^{j}$ 
coïncide avec son produit naïf par cette G-forme.
Cela donne 
\[ c_1(p^*\overline L|_X)^e = \sum_{j=0}^e \binom ej 
c_1(q^*\overline M|_Y)^{e-j} \wedge \widetilde\omega_L|_Y^{j} \]
Enfin, $\widetilde\omega_L^j=0$ si $j>r$
puisque $\widetilde\omega_L$ provient d'une $(1,1)$-forme sur~$T_\trop$ 
par~$\tau$.

On a vu que $\widetilde\omega_L$ descend à~$A$
en la G-forme~$\omega_L$,  si bien que $\widetilde\omega_L|_Y^j$
descend à~$X$ en la G-forme $\omega_L^j|_X$.
Les courants $c_1(p^*\overline L|_X)$ descend à~$X$, par construction.
Par suite, le courant $c_1(q^*\overline M|_Y)$
descend à~$X$, et chaque 
$c_1(q^*\overline M|_Y)^{i}$ descend en un courant~$S_i$ sur~$X$.
On a alors 
\[ c_1(\overline L|_X)^e 
 = \sum_{j=0}^e \binom ej S_{e-j} \wedge \omega_L|_X^j.\]
On a vu que $S_{e-j}=0$ si $e-j>b$, et que $\omega_L|_X^j=0$ si $j>r$;
par suite, il suffit de conserver ces courants~$S_j$
pour $j$ dans l'intervalle $[\sup(e-b,0),\inf(r,e)]$.
Cela conclut la démonstration de la première partie du corollaire.

La seconde assertion est une égalité de mesures;
il suffit de la prouver pour une fonction~$f$ qui est lisse.
En appliquant la proposition~\ref{prop.paralineaire-atlas}
aux courants $\binom dj S_{d-j}$ et aux G-formes $f\omega_L^j$,
on obtient des polyèdres calibrés~$P_j$, où $P_j$ est de dimension~$j$,
tels que $\langle \binom dj S_{d-j}, f\omega^j_L\rangle
= \int_{P_j} f\omega_L|_{P_j}^j$ pour tout~$j$, d'où la proposition.
\end{proof}

%
%

\section{Positivité des métriques canoniques}

Soit $A$ une $k$-variété abéloïde.

\begin{prop}\label{prop.canonique-positive}
Supposons que la valeur absolue de~$k$ n'est pas triviale.
Soit~$L$ un fibré en droites sur~$A$.
Reprenons les notations $E,p,q,B,M$ et $q_L$ introduites ci-dessus.
Les assertions suivantes sont équivalentes.
\begin{enumerate}\def\theenumi{\roman{enumi}}\def\labelenumi{(\theenumi)}
\item La métrique cubiste de~$L$ est psh;
\item On a $c_1(\overline L)\geq 0$;
\item La forme quadratique de~$q_L$ est positive 
et le fibré en droites résiduel~$\mathfrak M_s$ 
sur la $\tilde k$-variété abélienne~$\mathfrak B_s$ 
est numériquement effectif.

Lorsque $A$ est algébrique, elles sont équivalentes à la condition:
\item Le fibré en droites~$L$ est nef.
\end{enumerate}
\end{prop}
\begin{proof}
L'implication (i)$\Rightarrow$(ii) est standard.

Démontrons l'implication (iii)$\Rightarrow$(i).
Soit $y\in E$; comme le fibré résiduel de $q^*\overline M$ en~$y$
s'identifie au fibré résiduel de~$\overline M$ en~$q(y)$,
il est numériquement effectif. 
D'après le corollaire~\ref{coro.PL-psh-c1pos},
le fibré métrisé~$q^*\overline M$ est psh.
La fonction~$\tau$ est harmonique et strictement~\pl;
la fonction~$q_L$ est convexe sur~$\R^n$ puisque
c'est la somme d'une forme quadratique positive et
d'une forme affine.
D'après le corollaire~\ref{coro.psh-harmo},
la fonction $\tau^* q_L$ est psh sur~$E$. 
Puisque le fibré métrisé~$p^*\overline L$ 
s'identifie au produit de $q^*\overline M$ et du fibré
trivial muni de la métrique donné par la fonction $\tau^*q_L$,
il est psh. Comme $p$ a localement des sections,
le fibré métrisé~$\overline L$ est également psh.

Il reste à démontrer l'implication (ii)$\Rightarrow$(iii).
Prouvons d'abord que la partie quadratique de~$q_L$ est positive.
Soit $b$ un point rigide de~$B$; le fibré résiduel~$\overline M(b)$
est trivial, donc il existe un voisinage ouvert~$U$
de~$b$ dans~$B$ sur lequel le fibré métrisé~$\overline M$ est trivial.
L'extension~$E$ est G-localement trivialisable; 
puisque $b$ est un point
rigide, on peut également supposer qu'elle est trivialisable au-dessus
de~$U$. Il existe alors un isomorphisme $q^{-1}(U)\simeq U\times\gm^r$
qui identifie~$\tau$ à la tropicalisation du moment fourni
par la seconde projection.

Sur $q^{-1}(U)$, on a donc $c_1(p^*\overline L)=\tau^*(\frac12 \ddc q_L)$.
Puisque $c_1(\overline L)\geq 0$, il en est de même
de $c_1(p^*\overline L)$, si bien que $\tau^*(\frac12 \ddc q_L)$
est un courant positif.
Comme c'est une forme sur~$U$, cette forme est positive;
comme l'image de~$q^{-1}(U)$ par~$\tau$ est égale à~$\R^n$, la $(1,1)$-forme
$\ddc q_L$ sur~$\R^n$ est positive, ce qui signifie exactement
que la partie quadratique de~$q_L$ est positive.

Il reste à démontrer, sous l'hypothèse que $c_1(\overline L)$
est un courant positif,
que le fibré en droite résiduel~$\mathfrak M_s$ sur~$\mathfrak B_s$ 
est numériquement effectif.
D'après le corollaire~\ref{coro.PL-psh-c1pos},
il suffit de démontrer que le courant $c_1(\overline M)$ sur~$B$
est positif.
Soit $\alpha$ une forme lisse positive 
de type~$(n-1,n-1)$ sur~$B$;
démontrons que $\langle c_1(\overline M),\alpha\rangle \geq 0$.

Soit $\omega$ la $(1,1)$-forme standard sur~$\R^r$,
$\omega=\frac12\ddc \norm x^2= \sum_{i,j} \di x_i \wedge \dc x_j$.
Comme $\omega^r$ est positive, 
le courant $c_1(\overline L)\wedge \tau^*\omega^r$ est positif.
On a 
\[ c_1(\overline L)\wedge \tau^*\omega^r 
= c_1(q^*\overline M)\wedge \tau^*\omega^r 
+ \tau^*(\frac12\ddc q_L) \wedge \tau^* \omega^r
=  c_1(q^*\overline M)\wedge \tau^*\omega^r   \]
puisque le second terme, une $(r+1,r+1)$-forme issue de~$\R^r$, s'annule.
Ainsi, le courant $c_1(q^*\overline M)\wedge \tau^*\omega^r$
sur~$E$ est positif.

%

Soit $u$ une fonction lisse et positive, à support compact, sur~$E$.
La forme $u\cdot q^*\alpha$ est positive sur~$E$.
Puisque le courant $c_1(\overline M)\wedge \tau^*\omega^r$ est positif,
il vient
\[ \langle c_1(\overline M) \wedge \tau^*\omega^r , 
u \cdot q^*\alpha \rangle \geq 0. \]
On récrit cette inégalité comme
l'intégrale de la G-forme $u\cdot\tau^*\omega^r \wedge q^*\alpha$
contre le courant paralinéaire  $c_1(\overline M)$:
\[ \langle c_1(\overline M), u \cdot \tau^*\omega^r \wedge q^*\alpha \rangle
\geq 0 ,\]
et on applique la formule de projection 
du théorème~\ref{theo.formule-projection}.
Il vient donc
\[ \langle c_1(\overline M), v \alpha \rangle \geq 0, \]
où $v\colon B\to \R$ est définie par
\[ v(b) = \int_{E_b} u \tau^*\omega^r ,\]
puisque toutes les fibres de~$q$ sont de dimension~$r$. 

Par approximation uniforme, cette inégalité vaut
pour toute fonction continue positive~$u$ sur~$E$.
Considérons en particulier une fonction~$u$ de la forme $u_0\circ\tau$,
où $u_0$ est une fonction lisse à support compact sur~$\R^r$.
Alors, 
\[ v(b)=\int_{E_b} \tau^*u_0 \tau^*\omega^r
= \int_{\R^r} u_0 \omega^r, \]
est indépendant de~$b$,
et est strictement positif si $u_0\neq 0$.
Cela prouve que $\langle c_1(\overline M),\alpha\rangle\geq 0$
et conclut la démonstration de l'équivalence des
propriétés~(i), (ii) et~(iii).

Supposons maintenant que $A$ est algébrique
et démontrons qu'elles équivalent au caractère~nef de~$L$.

Supposons que $L$ est nef et démontrons que $c_1(\overline L)\geq 0$.
Écrivons $L^{\otimes 2}=(L\otimes [-1]^*L)\otimes (L\otimes [-1]^*L^{-1})$
comme le produit de deux fibrés en droites~$L_+$ et~$L_-$,
où $L_+$ est symétrique et $L_-$ est antisymétrique. 
En particulier, $L_-$ est algébriquement équivalent à~$0$,
de sorte que $c_1(\overline L_-)=0$.
Soit $N$ un fibré en droites ample et symétrique sur~$A$. 
Pour tout entier naturel~$t$, le fibré
en droites $L_+^{\otimes t}\otimes N$ sur~$A$ est alors ample,
en vertu du critère de Nakai--Moishezon.
Il possède en particulier des métriques psh; comme
il est symétrique, ses métriques cubistes sont obtenues
par le procédé du lemme~\ref{lemm.tate} et sont par conséquent psh. 
Ainsi, $t c_1(\overline L_+)+c_1(\overline N)\geq 0$.
En divisant par $t$ et en faisant tendre~$t$ vers~$+\infty$,
on obtient que $c_1(\overline L_+)\geq 0$. Par suite, $c_1(\overline L)\geq 0$.
Cela prouve l'implication (iv)$\Rightarrow$(ii) lorsque $A$ est algébrique.

Enfin, supposons~$\overline L$ psh et démontrons que $L$ est nef.
Soit $C$ une courbe irréductible dans~$A$.
D'après le théorème~\ref{theo.MA-lisse-masse},
on a $c_1(L)\cap [C] = \int_{C^\an} c_1(\overline L|_{C^\an})$.
Par ailleurs, $\overline L|_{C^\an}$ est psh,
de sorte que 
$c_1(\overline L|_{C^\an})\geq 0$. Par suite, $c_1(L)\cap [C]\geq 0$.
\end{proof}

\begin{coro}\label{coro.canonique-triviale}
Supposons que la valeur absolue de~$k$ n'est pas triviale.
Soit $L$ un fibré en droites sur~$A$.
Reprenons les notations $E,p,q,B,M$ et $q_L$ introduites ci-dessus.
Les assertions suivantes sont équivalentes.
\begin{enumerate}\def\theenumi{\roman{enumi}}\def\labelenumi{(\theenumi)}
\item Le fibré en droites~$L$ est invariant par translation;
\item On a $c_1(\overline L) = 0$;
\item On a $q_L=0$ et 
le fibré en droites résiduel~$\mathfrak M_s$ 
sur la $\tilde k$-variété abélienne~$\mathfrak B_s$ 
est numériquement trivial.

Lorsque $A$ est algébrique, elles sont équivalentes à la condition:
\item Le fibré en droites~$L$ est numériquement trivial.
\end{enumerate}
\end{coro}
\begin{proof}
Appliquée à~$L$ et~$L^{-1}$,
la proposition précédente
entraîne l'équivalence de~(ii), (iii), ainsi que~(iv) 
lorsque $A$ est algébrique.
Par ailleurs, il résulte de 
la proposition~\ref{prop.unif-pic0} que
les assertions~(i) et~(iii) sont équivalentes.
\end{proof}

\begin{rema}
Supposons que la valuation de~$k$ est triviale.
Dans ce cas, la transposition de la proposition~\ref{prop.canonique-positive}
et de son corollaire se heurte à une bonne théorie des modèles formels
d'espaces $k$-analytiques propres. Il ne semble par exemple pas connu,
bien que plausible, que toute variété abéloïde sur un corps trivialement
valué soit abélienne, et on se demande bien, si elle ne l'était pas,
ce que serait la fibre spéciale d'un éventuel modèle formel.

Cependant, si on se restreint aux variétés abéliennes, les propriétés~(i),
(ii) et~(iv) de la proposition~\ref{prop.canonique-positive} restent équivalentes.
Rappelons en effet (remarque~\ref{fibvaltriv}) 
que si $A$ est une variété abélienne et $L$ un fibré en droites
sur~$A$, alors
$(A,L)$ est essentiellement un modèle formel de~$(A^\an,L^\an)$,
si bien que  $L^\an$ possède une métrique formelle naturelle.
Cette métrique est cubiste, par le théorème du cube pour~$L$.
L'équivalence des propriétés~(i), (ii), (iv)
de la proposition~\ref{prop.canonique-positive}
pour $(A^\an,L^\an)$ 
est alors conséquence du corollaire~\ref{coro.PL-psh-c1pos}.

On en déduit également, toujours lorsque $A$ est une variété abélienne, que 
 les assertions~(i), (ii) et~(iv) du corollaire~\ref{coro.canonique-triviale}
pour $(A^\an,L^\an)$ sont équivalentes.
\end{rema}

\begin{rema}
Reprenons les notations du corollaire~\ref{coro.structure-mesure}.
Le calcul de la mesure $c_1(\overline L|X)^d$ 
fait a priori intervenir  des sous-espaces calibrés~$P_j$
de toutes dimensions~$j$
telles que $\sup(d-b, 0) \leq j\leq \inf(r,d)$.

a) La G-forme $\omega_L|_{P_j}$ se déduit de $\ddc(\frac12q_L)$
par l'application linéaire par morceaux $\tau|_{P_j}$.
C'est donc G-localement une (1,1)-forme constante sur~$P_j$.
Par suite, la mesure~$\mu_j(L)$ sur~$P_j$ déduite
de $\omega_L^{j}$ est 
G-localement multiple constant d'une mesure de Lebesgue sur~$P_j$.

Le support de cette mesure $\mu_j(L)$ est en particulier
un sous-espace polyédral de~$P_j$, purement de dimension~$j$, 
sur lequel $\tau$ est une immersion par morceaux.

b) Démontrons que si $L$ et $L'$ sont amples, alors
le support de $\mu_j(L)$ coïncide, pour tout~$j$, 
avec celui de $\mu_j(L')$. (Cette remarque n'a bien sûr d'intérêt
que si $A$ est algébrique…)

Il existe un entier~$p\geq 1$ tel que $L^p \otimes (L')^{-1}$
est ample. D'après la proposition~\ref{prop.canonique-positive}, 
on a $pc_1(\overline L|_X)\geq c_1(\overline L'|_X)$, ces deux courants étant positifs. On en déduit
que la mesure $c_1(\overline L|_X)^d$ domine un multiple
de la mesure $c_1(\overline L'|_X)^d$, et inversement par symétrie.
La description polyédrale de ces deux mesures, 
comme sommes des mesures $\mu_j(L)$ (resp. $\mu_j(L')$)
portée par des sous-espaces paralinéaires de dimensions distinctes, 
dont deux à deux orthogonales, entraîne l'égalité des supports.

c) Dans~\cite{gubler2010},
Gubler construit un exemple où les supports des mesures~$\mu_j(L)$
ne sont vides pour aucune valeur de~$j$ 
telle que $\sup(d-b, 0) \leq j\leq \inf(r,d)$.
\end{rema}

%
%
%
%
 
%

\clearpage
\appendix

\part*{Appendices}

\def\thechapter{\Alph{chapter}}\def\theHchapter{\Alph{chapter}}

\chapter{Compléments d'analyse et de topologie}

\section{Paracompacité} %

Commençons par rappeler quelques de définition de topologie générale.

\subsection{}
Soit $X$ un espace topologique.

Une famille $(A_i)_{i\in I}$ de parties de~$X$ est dite \emph{localement finie}
si tout point~$x$ de~$X$ possède un voisinage~$V$ tel que 
l'ensemble des~$i\in I$ tels que $V\cap A_i$ est non vide est fini.

Une famille $(B_j)_{j\in J}$ de parties de~$X$
est un \emph{raffinement} d'une famille~$(A_i)_{i\in I}$
si pour tout $j\in J$, il existe $i\in I$ tel que $B_j\subset A_i$.

Une famille $(A_i)_{i\in I}$ de parties de~$X$ est appelée
recouvrement si sa réunion est égale à~$X$,
et recouvrement fermé (resp.\ ouvert, etc.) si ses membres 
sont fermés (resp.\ ouverts, etc.).

\subsection{}
Un espace topologique~$X$ est dit \emph{paracompact}
s'il est séparé et si pour tout recouvrement ouvert de~$X$,
il existe un recouvrement ouvert de~$X$ qui le raffine
et qui est localement fini.

Un espace topologique compact est paracompact.

\subsection{}
Un espace topologique est dit \emph{dénombrable à l'infini}
s'il est réunion dénombrable de parties compactes.

Si $X$ est un espace topologique localement compact (ce qui suppose
qu'il est séparé, par définition), il est paracompact si et seulement
si il est somme disjointe d'espaces topologiques localement compacts
et dénombrables à l'infini. 
En particulier, un espace topologique localement compact et connexe
est paracompact si et seulement s'il est dénombrable à l'infini.

\begin{lemm}\label{comp.fort.paracomp}
Soit $X$ un espace topologique séparé, localement
compact et dont tout point possède une base d'ouverts dénombrables à l'infini.

\begin{enumerate}
\item
Toute partie compacte de~$X$
possède un voisinage ouvert qui est dénombrable à l'infini
et relativement compact.
\item
Toute partie dénombrable à l'infini de~$X$
possède un voisinage ouvert qui est dénombrable à l'infini,
et en particulier paracompact.
\end{enumerate}
\end{lemm}
\begin{proof}
Il suffit de démontrer la première assertion.
Soit $F$ une partie compacte de~$X$. Comme $X$ est localement
compact, $F$ possède donc un voisinage compact~$V$.
Soit $x$ un point de~$F$; il possède un voisinage ouvert~$U_x$ dans~$X$,
contenu dans~$V$  et dénombrable à l'infini.
Puisque~$F$ est compact, il existe une partie finie~$S$
de~$F$ telle que la réunion~$U$ des~$U_x$, pour $x\in S$, contienne~$F$.
Cet ensemble~$U$ est ouvert dans~$X$ et dénombrable à l'infini.
\end{proof}

L'énoncé suivant est un raffinement du théorème~5
de Bourbaki, Topologie générale, chapitre~I, \S9, \no10. 

\begin{lemm}\label{lemm.app-paracompact}
Soit $X$ un espace topologique paracompact et localement compact. 
Soit $\mathscr B$ un ensemble de parties de~$X$
dont les intérieurs forment une base de la topologie de~$X$.
Soit $(U_i)_{i\in I}$ un recouvrement ouvert de~$X$.
Il existe une famille $(V_j)_{j\in J}$ d'éléments de~$\mathscr B$, 
localement finie, raffinant~$(U_i)$, et telle
que la famille $(\mathring V_j)_{j\in J}$ recouvre~$X$.
\end{lemm}
\begin{proof}
On peut supposer que $V$ est connexe.
Il est alors dénombrable à l'infini
et on peut donc l'écrire comme réunion croissnte 
d'une suite~$(W_n)$ d'ensembles ouverts 
relativement compacts  telle que $\overline{W_n}\subset W_{n+1}$
pour tout~$n$.
Posons $K_0=\overline W_0$; pour tout entier~$n\geq 1$, posons 
$K_n=\overline W_n\setminus W_{n-1}$; les ensembles~$K_n$ sont compacts.
L'ensemble ouvert $\Omega_n=V_{n+1}\setminus \overline{V_{n-2}}$
est un voisinage de~$K_n$; 
par construction, $\Omega_n$ ne rencontre pas~$\Omega_m$ si $\abs{m-n}\geq 3$.
Pour tout $x\in K_n$, il existe un élément~$V_x$ de~$\mathscr B$
contenu dans un des ensembles~$U_i$ et contenu dans cet ouvert $\Omega_n$
et tel que $x\in\mathring V_x$.
Comme $K_n$ est compact, il existe une partie finie~$J_n$ 
de~$K_n$ telle que les $\mathring V_x$, pour $x\in J_n$, recouvrent~$K_n$.
La famille des~$V_x$, pour $n\in\N$ et $x\in J_n$, convient.
\end{proof}

\section{Séparation, densité et partitions de l'unité}

\subsection{}
Nous démontrons ici un certain nombre de résultats généraux
et relativement standard concernant 
la séparation, la densité de certains sous-espaces de fonctions
continues et les partitions de l'unité par de telles fonctions.

Nous considérons un espace topologique localement compact~$X$
(en particulier séparé) et un sous-faisceau~$\mathscr A$
en $\R$-algèbres
du faisceau des fonctions continues à valeurs réelles sur~$X$.
On suppose que pour tout ouvert~$U$ de~$X$,
tout élément $f\in \mathscr A(U)$ et toute fonction lisse~$\phi$
sur~$\R$ définie au voisinage de~$f(U)$, 
la fonction $\phi\circ f$ appartient à~$\mathscr A(U)$.

Notons que cette hypothèse est vérifiée dans les deux cas suivants :
\begin{itemize}
\item L'espace~$X$ est un espace topologique 
et $\mathscr A$ est le faisceau des fonctions continues sur~$X$;
\item L'espace~$X$ est un espace paralinéaire affiné
ou un espace analytique, plus généralement un espace G-tropical affiné,
et $\mathscr A$ est le faisceau des fonctions lisses sur~$X$.
\end{itemize}

Pour alléger la rédaction, nous appellerons fonctions lisses
les sections du faisceau~$\mathscr A$.

\subsection{}\label{ss.app-sep}
On dit que $\mathscr A(X)$ sépare les points si pour
tout couple~$(x,y)$ de points de~$X$ tels que $x\neq y$,
il existe une section globale~$f$ de~$\mathscr A_X$
telle que $f(x)\neq f(y)$.

On dit que $\mathscr A$ sépare localement les points
si pour tout point~$x$ de~$X$, il existe un voisinage ouvert~$U$
de~$X$ tel que $\mathscr A(U)$ sépare les points de~$U$.

\begin{lemm}
Supposons que $\mathscr A(X)$ sépare les points.
Soit $K$ un sous-ensemble compact de~$X$ et soit $x\in X\setminus K$.
Il existe une fonction lisse sur~$X$ qui vaut~$1$ au voisinage de~$x$
et nulle au voisinage de~$K$.
\end{lemm}
\begin{proof}
On commence par traiter le cas où $K$ est réduit à un point~$y$.
Par hypothèse,
il existe une fonction lisse~$f$ sur~$X$ telle que $f(x)\neq f(y)$.
On choisit une fonction lisse~$u$ sur~$\R$ qui est constante égale à~$0$
au voisinage de~$f(x)$ et constante égale à~$1$ au voisinage de~$y$;
l'existence de telles fonctions est classique.
La fonction $u\circ f$ sur~$X$ est lisse et convient.

D'après ce qui précède, pour tout $y\in K$, il existe une fonction lisse~$f_y$
sur~$X$ valant~$1$ au voisinage de~$x$ 
et nulle dans un voisinage ouvert~$V_y$ de~$y$. 
Soit~$S$ une partie finie de~$K$ telle que
les~$V_y$, pour $y\in S$, recouvrent~$K$;
la fonction $\prod_{y\in S} f_y$ convient.
\end{proof}

\begin{coro}\label{coro.app-lisse-separe}
On suppose que $\mathscr A$ sépare localement les points.
Soit $K$ une partie compacte de~$X$ et soit $U$ un voisinage ouvert de~$K$.
Il existe une fonction lisse sur~$X$ qui vaut~$1$ au voisinage de~$K$
et dont le support est compact et contenu dans~$U$.
\end{coro}
\begin{proof}
On commence par traiter le cas où le compact~$K$ est réduit à un point~$x$.
Soit $V$ un voisinage compact de~$x$ tel que $\mathscr A(\mathring V)$
sépare les points de~$\mathring V$.
L'ensemble~$F=V\setminus (\mathring V\cap U)$ est une partie fermée de~$X$
contenue dans~$V$. Soit $f$ une fonction lisse sur~$V$
qui vaut~$1$ au voisinage de~$x$ et qui est nulle au voisinage de~$F$.
Soit $g$ la fonction sur~$X$ qui coïncide avec~$f$ sur~$V$
et qui est nulle en dehors de~$V$. Elle est lisse sur l'intérieur de~$V$,
sur l'ouvert $X\setminus V$ ainsi que sur un voisinage ouvert de~$F$ dans~$X$;
ces trois ouverts recouvrent~$X$. La fonction~$g$ est donc lisse.
Son support est contenu dans~$U$; comme il est contenu dans~$V$,
il est aussi compact.

Traitons maintenant le cas général. Pour tout $x\in K$,
soit $f_x$ une fonction lisse sur~$X$ qui vaut~$1$ au voisinage de~$x$
et dont le support est compact et contenu dans~$U$.
Quitte à remplacer~$f_x$ par son carré, on suppose que cette fonction
est positive ou nulle. Comme $K$ est compact,
la somme~$g$ d'une sous-famille finie de ces fonctions est
strictement positive en tout point de~$K$.
Elle est lisse et son support est compact et contenu dans~$U$.
Soit $\phi$ une fonction lisse de~$\R$ vers~$[0;1]$,
nulle en~$0$ et égale à~$1$ sur un voisinage de~$g(K)$;
la fonction $f=\phi\circ g$ convient.
\end{proof}

\begin{theo}[Stone--Weierstraß] \label{theo.app-sw}
On suppose que $\mathscr A$ sépare localement les points.
Soit $U$ un ouvert de~$X$ et
soit $f$ une fonction continue à valeurs réelles sur~$X$
dont le support est compact et contenu dans~$U$.
Soit $a$ et $b$ des nombres réels  
tels que $a\leq f(x)\leq b$ pour tout $x\in X$.

Pour tout~$\eps>0$, il existe une fonction lisse~$g$ sur~$X$
dont le support est compact, contenu dans~$U$,
et telle que $a\leq g(x)\leq b$ et $\abs{f(x)-g(x)}<\eps$ pour tout $x\in X$.
\end{theo}
\begin{proof}
Soit $X'$ le compactifié d'Alexandroff de~$X$; comme $X$
est localement compact, c'est un espace topologique compact.
Soit $\mathscr S$ l'ensemble des fonctions réelles
sur~$X'$ qui sont la somme d'une fonction constante
et d'une fonction lisse à support compact sur~$X$ (prolongée par~$0$
au point à l'infini).
C'est une sous-algèbre  unitaire de l'algèbre des fonctions continues
sur~$X'$; elle sépare les points en vertu du corollaire~\ref{coro.app-lisse-separe}.
Prolongée par~$0$ au point à l'infini,
la fonction~$f$ définit une fonction continue~$f'$ sur~$X'$.
D'après le théorème de Stone--Weierstra\ss, il existe
un nombre réel~$c$ et une fonction lisse à support compact
sur~$X$, disons~$g$, tels que $\abs{f'-g-c}<\eps/4$.
En considérant cette inégalité au point à l'infini,
on a $\abs{c}<\eps/4$, si bien que $\abs{f-g}<\eps/2$ en tout point.

Soit $h\colon X\ra[0,1]$ une fonction lisse sur~$X$
dont le support est compact et contenu
dans~$U$ et qui vaut~$1$ au voisinage du support de~$f$.
La fonction $hg$ est lisse sur~$X$, son support est
compact et contenu dans~$U$.
Pour $x$ en dehors du support de~$f$, on a
$\abs{f(x)-h(x)g(x)}=\abs{h(x)g(x)}
\leq\abs{g(x)}=\abs{f(x)-g(x)}<\eps/2$.
Si $x$ appartient au support de~$f$,
on a $\abs{f(x)-h(x)g(x)}=\abs{f(x)-g(x)}<\eps/2$.

Considérons maintenant une fonction~$\mathscr C^\infty$
de~$\R$ dans~$\R$
vérifiant les conditions 
$\theta(x)=a$ pour $x\leq a$, $\theta(x)=b$ pour $x\geq b$,
$\abs{\theta(x)-x}\leq\eps/2$ pour $x\in [a;b]$ et,
si $a\leq 0\leq b$, $\theta(0)=0$.

Démontrons que la fonction $\theta\circ hg$ répond aux conditions
de l'énoncé.
Elle est lisse, prend ses valeurs dans l'intervalle~$[a,b]$
et, pour tout $x\in X$, on a 
\[ \abs{\theta\circ hg(x)-f(x)}\leq \abs{\theta \circ hg(x)-hg(x)}
+ \abs{hg(x)-f(x)} < \eps/2 + \eps/2 \leq \eps. \]

Il reste à vérifier que le support de~$\theta\circ hg$ 
est compact et contenu dans~$U$.
Si $X$ est compact et $X=U$, la condition à vérifier est vide.
Sinon,  $f$ s'annule, de sorte que $a\leq 0 \leq b$, d'où $\theta(0)=0$.
Par suite, le support de~$\theta\circ hg$ est contenu dans celui de~$g$,
il est donc compact et contenu dans~$U$.
\end{proof}

\begin{coro}\label{coro.app-sw}
On suppose que $\mathscr A$ sépare localement les points.

\begin{enumerate}
\item L'ensemble des fonctions lisses à support compact sur~$X$
est dense dans l'espace~$\mathscr C_\cpct(X)$
des fonctions continues à support compact sur~$X$, 
muni de la topologie limite inductive des espaces
de fonctions à support dans un compact donné.
\item L'ensemble des fonctions lisses à support compact sur~$X$
est dense dans l'espace~$\mathscr C(X)$ des fonctions continues sur~$X$
muni de la topologie de la convergence uniforme sur tout compact.
\end{enumerate}
\end{coro}
\begin{proof}
\begin{enumerate}
\item
Soit $\Omega$ une partie ouverte non vide de~$\mathscr C_\cpct(X)$;
soit $K$ une partie compacte de~$X$ telle que $\Omega$
rencontre~$\mathscr C_K(X)$ et soit $f\in\Omega\cap\mathscr C_K(X)$.
Soit $V$ un voisinage ouvert de~$K$ dans~$X$, relativement compact.
Par définition de la topologie de~$\mathscr C_\cpct(X)$,
$\Omega\cap\mathscr C_{\bar V}(X)$ est ouvert dans~$\mathscr C_{\bar V}(X)$.
Il existe donc $\eps>0$ tel que la boule de~$\mathscr C_{\bar V}(X)$,
de centre~$f$ et de rayon~$\eps$, soit contenue dans~$\Omega$.
D'après le théorème~\ref{theo.app-sw}, cette boule contient
une fonction lisse.

\item
Soit $\Omega$ une partie ouverte non vide de~$\mathscr C(X)$
et soit $f\in\Omega$. Soit $K$ une partie de~$X$ et $\eps>0$;
démontrons qu'il existe une fonction lisse à support compact sur~$X$,
$g$, telle que $\norm{g-f}_K\leq \eps$.
Soit $V$ un voisinage relativement compact de~$K$ de~$X$
et soit $\lambda$ une fonction lisse à support contenu dans~$V$
qui vaut~$1$ sur~$K$
D'après le théorème~\ref{theo.app-sw}, il existe une fonction~$g$,
lisse à support contenu dans~$V$, telle que $\norm{g-\lambda f}\leq \eps$.
En particulier, $\norm{g-f}_K\leq\eps$.\qedhere
\end{enumerate}
\end{proof}

\subsection{}
Soit $X$ un espace topologique, soit $(U_i)_{i\in I}$
un recouvrement ouvert de~$X$. 
Une partition de l'unité\index{partition de l'unité} 
subordonnée au recouvrement~$(U_i)$ est une famille
localement finie~$(f_i)_{i\in I}$
de fonctions continues sur~$X$
telle que pour tout~$i\in I$, le support de~$f_i$ soit 
contenu dans~$U_i$. La mention « localement finie »
signifie que tout point~$x\in X$ possède un voisinage~$V$
sur lequel seul un nombre fini des~$f_i$ ne sont pas
identiquement nulles.

Si $X$ est paracompact, il existe une partition de l'unité
subordonnée à tout recouvrement ouvert de~$X$
(Bourbaki, \emph{Topologie générale}, IX, \S4, \no4, p.~40,
corollaire~1). 

La remarque suivante nous permet, pour construire
une partition de l'unité subordonnée à un recouvrement,
de le raffiner autant que nécessaire.
Soit $(U_i)_{i\in I}$ et $(V_j)_{j\in J}$ deux recouvrements
ouverts de~$X$. Supposons le recouvrement~$(V_j)$
plus fin que le recouvrement~$(U_i)$, au sens où pour tout~$j\in J$,
il existe $i\in I$ tel que $V_j\subset U_i$.
Soit alors $(g_j)_{j\in J}$ une partition de l'unité
subordonnée au recouvrement~$(V_j)$.
Pour tout $j\in J$, soit $a(j)$ un élément de~$I$
tel que $V_j\subset U_{a(j)}$.
Pour tout $i\in I$, posons $f_i=\sum_{a(j)=i} g_j$;
la famille $(f_i)_{i\in I}$ est une partition de l'unité
subordonnée au recouvrement~$(U_i)_{i\in I}$.

\begin{prop}[Partitions de l'unité lisses]
\label{prop.app-partition.1.lisse}
On suppose que $X$ est paracompact et que $\mathscr A$ sépare localement
les points.
Soit $(U_i)_{i\in I}$ un recouvrement ouvert de~$X$.
Il existe une partition de l'unité\index{partition de l'unité} 
subordonnée à~$(U_i)$
formée de fonctions lisses sur~$X$.
\end{prop}
Lorsque $\mathscr A$ est le faisceau des fonctions continues,
on retrouve ainsi l'énoncé classique qu'un espace paracompact 
admet des partitions de l'unité. 
Bien sûr, notre démonstration s'inspire de celle-là.
Pour une réciproque partielle qui indique la nécessité
de l'hypothèse de paracompacité, 
voir la proposition~\ref{prop.partition-paracompact}.

\begin{proof}
On peut supposer que $X$ est connexe.
Comme $X$ est paracompact et localement compact,
il existe un recouvrement localement  fini formé d'ouverts
relativement compacts dans~$X$ qui raffine le recouvrement~$(U_i)$.
Cela nous autorise à supposer que le recouvrement~$(U_i)$
est déjà de cette forme.
Puisque $X$ est connexe et paracompact, il est
dénombrable à l'infini ; on peut donc indexer par 
une partie de l'ensemble des entiers naturels 
l'ensemble d'indices du recouvrement~$(U_i)$.

Considérons une partition de l'unité continue $(g_i)$ subordonnée 
au recouvrement~$(U_i)$. 
Pour tout~$i$, le support de~$g_i$ est compact.
Il existe donc une fonction lisse~$h_i$ sur~$X$, à support compact
contenu dans~$U_i$, telle que $\abs{h_i-g_i}<1/2^{i+2}$.
Posons $h=\sum h_i$; c'est la somme d'une famille 
localement finie donc fonctions lisses, donc c'est une fonction lisse.
Pour $x\in X$, 
\[ h(x) = \sum_{i} h_i(x)  
\geq \sum_i g_i(x) - \sum_i 2^{-i-2} \geq 1 - \sum_{i=0}^\infty 2^{-i-2}
= \dfrac12. \]

Posons $f_i=h_i/h$; c'est une fonction lisse sur~$X$
dont le support est contenu dans~$U_i$.  On a $\sum f_i=1$.
\end{proof}

\begin{coro}
\label{prop.app-trop-partition.1.lisse-compact}
Soit $F$ une partie fermée de~$X$.
Supposons que $F$ possède un voisinage ouvert~$U$
qui est paracompact et tel que $\mathscr A|_U$ sépare localement
les points.
Soit $(U_i)_{i\in I}$ un recouvrement de~$F$ par des ouverts de~$X$.
Il existe une famille localement finie $(\theta_i)$ 
de fonctions lisses positives sur~$X$,
telle que pour tout~$i$, le support de~$\theta_i$ soit 
compact et contenu dans~$U_i$
et telle que $\sum\theta_i$ vaille~$1$ au voisinage de~$F$.
\end{coro}
\begin{proof}
Soit $U$ un voisinage ouvert de~$F$ dans~$X$ qui est paracompact
tel que $\mathscr A|_U$ sépare localement les points.
Soit $(V_j)_{j\in J}$ un recouvrement ouvert de~$F$
par des ouverts relativement compacts
de~$U$ qui est plus fin que le  recouvrement~$(U_i)_{i\in I}$.
Choisissons un élément~$0$ qui n'appartient pas à~$J$ ;
posons $J'=J\cup\{0\}$ et $V_0=U\setminus F$. 
D'après la proposition~\ref{prop.app-partition.1.lisse}
appliquée à l'espace paracompact~$U$ 
au faisceau~$\mathscr A|_U$
et au recouvrement ouvert~$(V_j)_{j\in J'}$,
il existe une partition de l'unité $(g_j)_{j\in J'}$
formée de fonctions lisses sur~$U$
qui est subordonnée  au recouvrement ouvert~$(V_j)_{j\in J'}$. 

Pour tout $j\in J$, soit $g'_j$ la fonction sur~$X$ qui coïncide avec~$g_j$
sur~$V_j$ et est identiquement nulle en dehors de~$V_j$.
Comme le support de~$g_j$ est compact et contenu dans~$V_j$,
il est fermé dans~$X$. Par suite, la fonction~$g'_j$ est lisse sur~$X$.

Pour tout $j\in J$, choisissons un élément $a(j)\in I$
tel que $V_j\subset U_{a(j)}$. Pour tout $i\in I$,
posons $f_i=\sum_{a(j)=i} g'_j$.
La famille $(f_i)_{i\in I}$ convient.
\end{proof}

\begin{coro}\label{coro.app-fin}
On suppose que $X$ est paracompact.
Si le faisceau  $\mathscr A$ sépare
localement les points, il est fin.
\end{coro}

\begin{prop}\label{prop.partition-paracompact}
Soit $X$ un espace topologique séparé dont les composantes
connexes sont ouvertes, par exemple connexe ou localement connexe.
Soit $(f_i)$ une partition de l'unité sur~$X$
formée de fonctions continues à supports compacts.
Alors, $X$ est localement compact et paracompact.
\end{prop}
\begin{proof}
Pour tout~$i$, soit~$U_i$ l'ensemble des~$x\in X$ tels 
que $f_i(x)>0$, soit~$F_i$ son adhérence; par hypothèse, $F_i$ est compact.
Comme $\sum f_i=1$, les ouverts~$U_i$ recouvrent~$X$;
cela entraîne que $X$ est localement compact.

L'espace~$X$ est somme  de ses composantes connexes, car elles
sont ouvertes.
Pour démontrer qu'il est paracompact, nous pouvons
donc supposer que $X$ est connexe et non vide; 
il s'agit alors de démontrer
qu'il est dénombrable à l'infini.

Soit~$\sim$ la relation d'équivalence sur~$X$ la plus fine
telle que $x\sim y$ s'il existe $i\in I$ tel que $\{x,y\}\subset U_i$.
Comme les~$U_i$ sont ouverts, les classes d'équivalences sont ouvertes.
Comme $X$ est connexe et non vide, il n'y a qu'une seule classe d'équivalence.
Soit maintenant~$x\in X$ et soit $i\in I$. 
Construisons par récurrence une suite croissante $(I_n)$ de parties finies
de~$I$. On pose $I_0=\{i\}$; supposons que $I_n$ est défini.
Soit $F$ la réunion des~$F_i$ pour $i\in I_n$; c'est une partie compacte
de~$X$; soit $I_{n+1}$ l'ensemble des~$i\in I$ tels 
que $F\cap U_i\neq\emptyset$. Il contient~$I_n$.
Comme le recouvrement~$(U_i)$ est localement fini,
$I_{n+1}$ est fini. Par hypothèse, tout point~$y$ de~$X$
appartient à l'un des~$F_i$, pour $i\in\cup_{n\geq 0} I_n$.
Cela prouve que $X$ est réunion dénombrable de parties compactes
de~$X$, ainsi qu'il fallait démontrer.
\end{proof}

\section{Rappels de quelques faits de théorie de la mesure}

\subsection{}
On ne considère que des espaces localement compacts (donc, implicitement,
séparés), et localement connexes par arcs (si besoin),
mais on ne peut pas faire d'hypothèse  supplémentaire.
En particulier, ils ne sont pas tous paracompacts (toute
composante est dénombrable à l'infini, ici).
Exemples de la droite projective et du plan projectif sur $\C((t))$ 
privés du point de Gauß; ils ne sont  pas $\sigma$-compacts,
le second  étant connexe, n'est même pas paracompact.

\subsection{}
Soit $X$ un espace localement compact. 
La théorie de l'intégration sur un tel espace
donne lieu à deux points de vue concurrents
et nos arguments nous obligent apparemment à utiliser l'un
et l'autre des points de vue.

Nous appelerons \emph{mesure de Borel} (positive), ou plus simplement
\emph{mesure} (positive),
une fonction $\sigma$-additive
à valeurs dans~$\R_+\cup\{\infty\}$ sur la  tribu  de Borel.
Une telle mesure sera dite \emph{localement finie} (\resp \emph{finie})
si elle prend une valeur finie en tout compact (\resp sur~$X$).
Elle sera dite \emph{régulière} si la mesure d'une partie
borélienne est la borne supérieure des mesures de ses
sous-ensembles compacts.
Une \emph{mesure signée} est la différence de deux mesures finies;
une mesure signée est dite régulière si ses parties positive
et négative le sont.

On appelle \emph{mesure de Radon} sur~$X$ une forme linéaire continue
sur l'espace vectoriel~$\mathscr C_c(X)$ 
des fonctions continues à support compact sur~$X$, 
muni de la topologie limite inductive déduite de la filtration
par le support. On déduit du lemme d'Urysohn
que pour qu'une forme linéaire~$\mu$ sur
cet espace soit continue il suffit que l'on ait $\mu(f)\geq 0$
pour toute fonction $f\in\mathscr C_c(X)$ qui est positive;
on parle alors de mesure de Radon positive. 
Toute mesure de Radon sur~$X$  est différence de deux mesures
de Radon positives (\cite{INT-1-4}, chapitre~3, p.~53, th.~2).

\subsection{}
Soit $\mu$ une mesure de Borel positive sur~$X$.

Une fonction $f\colon X\to\R\cup\{\pm\infty\}$ est appelée fonction-plateau
s'il existe une partition dénombrable $(A_i)$ de~$X$
par des parties boréliennes
telle que $f$ soit constante de valeur~$a_i$ sur chaque~$A_i$.
Si $f$ est une fonction borélienne à valeurs dans~$[0,+\infty]$,
on pose  $\mu(f)=\sum a_i \mu(A_i)$,
avec la convention $(+\infty)\times 0=0\times(+\infty)=0$;
c'est un élément de $[0,+\infty]$. 
Si $f\colon X\to[0,+\infty]$
est une fonction positive arbitraire, on définit $\mu^*(f)$
comme la borne supérieure des intégrales $\mu(g)$, où $g$ est une fonction
plateau telle que $0\leq g \leq f$; on dit que $f$ est intégrable
si $\mu^*(f)<\infty$.

On dit qu'une fonction borélienne $f\colon X\to[-\infty,+\infty]$ 
est intégrable 
si les deux fonctions (boréliennes)
$f_+=\max(0,f)$ et $f_-=-\max(0,-f)$ sont intégrables
et l'on pose alors $\mu(f)=\mu(f_+)-\mu(f_-)$.
Si $f$ est positive, on a $f_-=0$, donc $\mu(f)=\mu^*(f)$.
On note $\mathscr L^1(X,\mu)$ l'ensemble des fonctions boréliennes
intégrables. L'application $f\mapsto \mu(f)$ sur $\mathscr L^1(X,\mu)$
est une forme linéaire (\cite{rudin1987}, Théorème~1.32).

\subsection{Théorème de représentation de Riesz}
Si $\mu$ est une mesure de Borel localement finie (positive),
alors l'espace $\mathscr C_c(X)$ est contenu dans $\mathscr L^1(X,\mu)$.
Toute mesure de Borel localement finie donne ainsi lieu 
à une mesure de Radon positive.
Par suite, une mesure signée donne lieu à une mesure de Radon sur~$X$.

Inversement, toute mesure de Radon positive sur~$X$,
est associée à une unique mesure de Borel localement
finie et \emph{régulière} (\cite{pollard-topsoe1975}, exemple~3,
p.~183; de nombreux ouvrages, par exemple~\cite{rudin1987},
utilisent d'autres régularités et ne garantissent pas l'unicité.)
On en déduit que toute forme  linéaire sur $\mathscr C_c(X)$
qui est continue pour la norme uniforme est associée à une mesure signée.



\subsection{}
Soit $f\colon Y\to X$ une application borélienne.
Soit $\mu$ une mesure  de Borel sur~$Y$. On définit alors
une mesure de Borel $f_*\mu$  sur~$X$
par la formule $f_*\mu(A)=\mu(f^{-1}(A))$ pour toute partie
borélienne~$A$ de~$X$.

Supposons que $f$ est continue;
démontrons alors que $f_*\mu$ est régulière si $\mu$ l'est.
Soit $A$ une partie borélienne de~$X$, soit $r$ un nombre
réel tel que $r<f_*\mu(A)$,  soit $K$ une partie
compacte de~$Y$, contenue dans~$f^{-1}(A)$,  telle que $\mu(K)>r$.
Alors, $f(K)$ est une partie compacte de~$A$ telle
que $f_*\mu(f(K))=\mu(f^{-1}(f(K)))\geq \mu(K)>r$, d'où l'assertion.

Si $f$ est continue et compacte et si $\mu$ est localement
finie, alors $f_*\mu$ est localement finie.

\subsection{}
Soit  $f\colon Y\to X$ une application continue et compacte.
Soit $\mu$ une mesure de Radon sur~$Y$. Pour toute fonction continue
à support compact~$\phi$ sur~$X$, la fonction $\phi\circ f$
sur~$Y$ est continue et à support compact; on définit ainsi une
mesure de Radon $f_*\mu$ sur~$X$ en posant $(f_*\mu)(\phi)=\mu(\phi\circ f)$.

Cette notion coïncide avec l'image des mesures de Borel
lorsque $\mu$ est positive.

\subsection{}
Comme $X$ est localement compact, une
partie~$A$ de~$X$ est localement compacte si et seulement
si elle est localement fermée (intersection d'un ouvert et d'un fermé).

Soit $\mu$ une mesure de Borel sur~$X$ 
et soit $A$ une partie localement compacte de~$X$.
En restreignant~$\mu$ à la tribu des parties boréliennes de~$A$,
on obtient une mesure sur~$A$ appelée, par abus,
la restriction de~$\mu$ à~$A$, et notée $\mu_A$ ou $\mu|_A$.
Si $\mu$ est localement finie (resp. si $\mu$ est régulière), 
il en est de même de~$\mu_A$.

Ainsi, les mesures de Borel positives localement
finies régulières forment 
de façon naturelle un préfaisceau 
sur la catégorie des ensembles localement compacts de~$X$.

D'après le théorème de Riesz, les mesures de Radon positives
forment donc aussi un préfaisceau sur  cette catégorie.
On en déduit qu'il en est de même
des mesures de Radon (non nécessairement positives).

Soit $A$ une partie localement compacte de~$X$ et soit
$f\in\mathscr C_c(X)$ une fonction continue dont le support
est contenu dans~$A$.  Pour toute mesure de Radon~$\mu$
sur~$X$, on a $\mu_A(f|_A)=\mu(f)$.
En effet, l'intégrale d'une fonction
continue se déduit de la connaissance 
de celle des fonctions caractéristiques des boréliens
contenus dans son support.

%
%
%

\subsection{}\label{recollement-Borel}
La topologie de Borel  sur l'espace topologique~$X$
est une topologie de Grothendieck ensembliste sur la « catégorie »
des parties localement fermées de~$X$ : une famille $(A_i)_{i\in I}$
est un B-recouvrement de~$U$ 
si pour tout~$x\in U$, il existe une partie finie~$J$ de~$I$
telle que $x\in A_j$ pour tout $j\in J$,
et la réunion des~$A_j$, pour $j\in J$, soit un voisinage de~$x$.
Considérons un B-recouvrement $(A_i)_{i\in I}$ de~$U$;
on dit qu'une partie~$V$ de~$U$ est adaptée au recouvrement $(A_i)$
s'il existe
une partie finie $J\subset I$ tels que pour tout $j\in J$,
$V\cap A_j$ soit fermé dans~$V$ et $V=\bigcup_{j\in J}(V\cap A_j)$.
Tout sous-ensemble d'une partie adaptée l'est encore.

En outre, tout point de~$X$ possède une base de voisinages
compacts adaptés au B-recouvrement $(A_i)_{i\in I}$.
Partons en effet d'une partie finie~$J$ comme dans la définition;
pour tout $j\in J$, soit $V_j$ un voisinage de~$x$ dans~$U$
tel que $A_j\cap V_j$ soit fermé dans~$V_j$. 
Posons
\[ V = \left(\bigcup_{j\in J} A_j\right) \cap \bigcap_{j\in J} V_j. \]
C'est un voisinage de~$x$ dans~$U$, car $J$ est finie,
et tout voisinage de~$x$ qui est contenu dans~$V$ convient.
En particulier, il existe de tels voisinages~$V$ compacts.

\begin{lemm}\label{lemm.app-raffinement}
Toute partie compacte~$V$ de~$X$ admet un recouvrement
fini par des parties compactes 
qui est plus fin que le B-recouvrement~$(A_i\cap V)$.
\end{lemm}
\begin{proof}
Pour $x\in V$, considérons  $V_x$ un voisinage de~$x$ dans~$X$,
compact, et adapté au recouvrement~$(A_i)$; soit $J_x$
une partie finie de~$I$ telle que la famille $(A_j\cap V_x)_{j\in J_x}$ 
soit un recouvrement fermé de~$V_x$.
Soit $S$ une partie finie de~$V$ telle que les~$V_x$, pour $x\in S$,
recouvrent~$X$, et soit $J$ la réunion des~$J_x$, pour $x\in S$.
Les ensembles $A_j\cap V_x\cap V$, pour $x\in S$ et $j\in J$,
forment un recouvrement fini de~$V$ par des parties compactes.
\end{proof}

\subsection{}
Les mesures 
forment un faisceau  sur la catégorie des ensembles boréliens
(munie de la topologie discrète, pour laquelle les recouvrements
sont les familles $(A_i)$ d'ensembles boréliens dont
une sous-famille dénombrable recouvre~$X$).
En effet, si $(A_i)_{i\in I}$ est une telle famille et si, pour tout $i$,
$\mu_i$ est une mesure sur~$A_i$, telles que les restrictions
à~$A_i\cap A_j$ de~$\mu_i$ et $\mu_j$ coïncident, si $i\neq j$,
alors il existe une unique mesure~$\mu$ sur~$X$ dont la restriction
à~$A_i$ est égale à~$\mu_i$.
(Remplacer $I$ par~$\N$, puis $A_i$ par le complémentaire $B_i$
de la réunion $\bigcup_{j<i}A_j\cap A_i$, ce qui fournit
une partition borélienne $(B_i)$ dénombrable de~$X$;
remplacer $\mu_i$ par sa restriction à~$B_i$.)

\begin{lemm}\label{lemm.recollement-compact}
Soit $X$ un espace topologique localement compact.
Pour toute partie compacte~$K$ de~$X$,
supposons donnée une mesure de Borel régulière~$\mu_K$
sur~$K$; on suppose que pour tout couple $(K,L)$
de parties compactes de~$X$ telles que $L\subset K$,
on a $\mu_K|_L=\mu_L$.
Alors, il existe une unique mesure de Borel régulière~$\mu$
sur~$X$ telle que $\mu|_K=\mu_K$ pour toute partie compacte~$K$.
\end{lemm}
\begin{proof}
Soit $\mu$ une mesure de Borel régulière sur~$X$ dont la trace
sur~$K$ est égale à~$\mu_K$ pour toute partie compacte~$K$ de~$X$.
Pour toute partie borélienne~$A$ de~$X$, 
on a nécessairement
\[ \mu(A)  = \sup_{\substack{K\subset A\\\text{$K$ compact}}}
  \mu(K) = \sup_{\substack{K\subset A \\ \text{$K$ compact}}} \mu_K(K). \]
Cela prouve qu'il existe au plus une telle mesure. Pour établir
son existence, définissons $\mu(A)$  
par la formule précédente,
pour toute partie borélienne~$A$ de~$X$.

Soit $K$ une partie compacte de~$X$ et soit $A$ 
une partie borélienne de~$X$ contenue dans~$K$. Par définition,
on a 
\[ \mu(A) =\sup_{\substack{L\subset A \\ \text{$L$ compact}}} \mu_L(L)
 = \sup_{\substack{L\subset A \\ \text{$L$ compact}}} \mu_K(L) 
 = \mu_K(A), \]
car $\mu_K$ est une mesure de Borel régulière.

D'autre part, $\mu$ est une fonction d'ensemble positive et
croissante.  Démontrons qu'elle est dénombrablement additive.
Soit $(A_n)$ une famille dénombrable de  parties boréliennes disjointes,
notons $A$ sa réunion et prouvons que $\mu(A)=\sum \mu(A_n)$.
Soit $K$ une partie compacte de~$A$. On a 
\[ \mu(K) = \mu_K(K) = \sum \mu_K(K\cap A_n) =\sum \mu(K\cap A_n)
 \leq \sum \mu(A_n). \]
On a donc $\mu(A)\leq \sum \mu(A_n)$.
Inversement, traitons le cas où $\sum \mu(A_n)<\infty$. 
Soit $\eps>0$. Soit $N$ tel que $\sum_{n=1}^N \mu(A_n)\geq \sum\mu(A_n)-\eps$.
Pour tout $n\in\{1,\dots,N\}$, soit $K_n$ une partie compacte
de~$A_n$ telle que $\mu(K_n)\geq \mu(A_n)-2^{-n}\eps$
et posons $K=K_1\cup\dots\cup K_N$.
Alors, $K$ est une partie compacte de~$A$ telle que 
\[ \mu_K(K) = \sum_{n=1}^N \mu_K(K_n) =\sum_{n=1}^N \mu(K_n)
\geq \sum_{n=1}^N \mu(A_n)- \eps \geq \sum \mu(A_n)-2\eps. \]
Par suite, $\mu(A)\geq \sum \mu(A_n)$. Le cas $\sum\mu(A_n)=\infty$
est analogue.
\end{proof}

\begin{prop}\label{prop.mesures-B-faisceau}
Les mesures de Borel régulières,
les mesures de Radon positives et
les mesures de Radon
forment des faisceaux pour la B-topologie sur~$X$.
En particulier, elles forment des faisceaux
pour la topologie usuelle.
\end{prop}
\begin{proof}
Soit $(A_i)$ un B-recouvrement de~$X$ et,
pour tout $i$, soit  $\mu_i$ une mesure de Borel régulière sur~$A_i$;
on suppose que pour tout couple~$(i,j)$, 
on a $\mu_i|_{A_i\cap A_j} =\mu_j|_{A_j\cap A_i}$.
Démontrons qu'il existe une unique mesure régulière~$\mu$ sur~$X$
telle que $\mu|_{A_i}=\mu_i$ pour tout $i$.

Traitons d'abord le cas où $X$ est compact.
Quitte à raffiner le B-recouvrement~$(A_i)$, on peut supposer
qu'il s'agit d'un recouvrement fini par des parties compactes
(lemme~\ref{lemm.app-raffinement}).
Si $\mu_i$ était localement finie pour tout~$i$, cela permet
de les supposer finies.
Soit $J$ une partie finie de~$I$
telle que  $A_j$ soit fermé pour tout $j\in J$, et les $A_j$, pour $j\in J$,
recouvrent~$X$.   

D'après le paragraphe~\ref{recollement-Borel},
il existe une unique mesure borélienne~$\mu$ sur~$X$ 
dont la restriction à~$A_i$
est égale à~$\mu_i$, pour tout $i\in I$. 
Démontrons que cette mesure est régulière.
Soit en effet une partie borélienne~$A$ de~$X$.
Soit $\eps>0$. 
Pour tout $i\in I$, soit $C_i$ une partie compacte
de~$A\cap A_i$ telle que $\mu_i(C_i)> \mu_i(A\cap A_i)-\eps$.
Soit $C$ la réunion des~$C_i$; c'est une partie compacte de~$A$
et l'on a 
\[ \mu(A\setminus C) \leq \sum_{i\in I} \mu ((A\setminus C)\cap A_i)
 \leq \sum_{i\in I} \mu_i(A_i\setminus C_i)\leq \eps\Card(I), \]
donc est arbitrairement petit.
On a $\mu|_{A_i}=\mu_i$ pour tout~$i$;
comme $X$ est recouvert par la famille finie~$(A_i)_{i\in I}$,
la mesure~$\mu$ est la seule mesure de Borel 
sur~$X$ qui vérifie cette propriété. 
Si les mesures~$\mu_i$ sont localement finies,
la mesure~$\mu$ est (localement) finie.

Revenons maintenant au cas général en ne supposant plus que $X$ est compact.
D'après ce qui précède, il existe,
pour toute partie compacte~$V$ de~$X$, une unique mesure
de Borel régulière~$\mu_V$
sur~$V$ telle que $\mu_V|_{A_i\cap V}
=\mu_i|_{A_i\cap V}$ pour tout~$i\in I$.
Si $V\subset W$, il résulte des définitions que $\mu_W|_V=\mu_V$.
D'après le lemme~\ref{lemm.recollement-compact},
il existe alors une unique mesure de Borel régulière~$\mu$ sur~$X$
telle que $\mu|_V=\mu_V$ pour toute partie compacte~$V$ de~$X$.
Si les mesures~$\mu_V$ sont (localement) finies, la mesure
$\mu$ construite est localement finie, car $X$ est localement compact.

Soit $i\in I$.
Démontrons que la restriction de~$\mu$ à~$A_i$ est égale à~$\mu_i$.
Par définition d'une mesure de Borel régulière, il suffit
de vérifier que pour toute partie compacte~$C$ de~$A_i$,
on a $\mu|_C=\mu_i|_C$, ce qui est évident puisqu'alors
$\mu|_C=\mu_C=\mu_C|_{C\cap A_i}=\mu_i|_{C\cap A_i}$.

Soit $\nu$ une mesure de Borel régulière 
telle que $\nu|_{A_i}=\mu_i$ pour tout~$i$.
Pour toute partie compacte~$V$ de~$X$, on a donc
$\nu|_{A_i\cap V}=\mu_i|_{V}=\mu_V|_{A_i\cap V}$,
donc $\nu|_V=\mu_V$ par définition de~$\mu_V$. Donc $\nu=\mu$.

Nous avons ainsi prouvé que les mesures de Borel régulières
forment un faisceau pour la B-topologie, et que les mesures localement
finies en constituent un sous-faisceau.

Le cas des mesures de Radon s'en déduit en les décomposant
en la différence de leurs parties positive et négative.
\end{proof}

%
%

\subsection{}\label{app-image-inverse-mesure}
Soit $f\colon Y\to X$ une application continue.
On suppose qu'il existe un B-recouvrement $(Y_i)$ de~$Y$
tel que, pour tout~$i$, $f|_{Y_i}$ soit un homéomorphisme de~$Y_i$
sur une partie $f(Y_i)$ de~$X$. Dans ce cas, $f(Y_i)$ est localement
compact, donc localement fermé dans~$X$;
on pourrait ainsi dire que $f$ est un « homéomorphisme B-local ».

Pour toute mesure de Borel (\resp de Radon) $\mu$ sur~$X$,
il existe alors une unique mesure de Borel (\resp de Radon) $f^*\mu$
sur~$Y$ dont la restriction à~$Y_i$ est égale à $(f^{-1})_*(\mu|_{f(Y_i)})$.
Cette mesure ne dépend pas du choix du B-recouvrement~$(Y_i)$.

\def\conv{\operatorname{conv}}

\section{Rappels d'analyse convexe}

Pour la commodité du lecteur, nous rappelons ici
quelques faits élémentaires d'analyse convexe, 
en tâchant  d'en donner les démonstrations.
L'ouvrage~\cite{hiriart-urruty-lemarechal2001} nous
a été une référence très utile.

Par \emph{espace affine}, nous entendons ici un espace
affine de dimension finie sur~$\R$. On désignera alors
par $\norm\cdot$ une norme sur l'espace vectoriel sous-jacent.

\subsection{}
Soit $V$ un espace affine.
Si $x$ et $y$ sont des points de~$V$, on note
$[x,y]$, $\mathopen]x,y]$, $[x,y\mathclose[$ et $\mathopen]x,y\mathclose[$
les segments (respectivement fermé, semi-ouverts et ouvert) 
d'extrémités~$x$ et~$y$:
ce sont les ensembles des points de~$V$ de la forme $(1-t)x+ty$,
où $t$ parcourt respectivement les intervalles $[0,1]$, $\mathopen]0,y]$, $[0,1\mathclose[$ et $\mathopen]0,1\mathclose[$.

\subsection{}
Une partie~$C$ de~$V$ est convexe si pour tout couple $(x,y)$
de points de~$C$, le segment $[x,y]$ est contenu dans~$C$.

L'intersection d'une famille partie de  convexes est convexe.
L'enveloppe convexe~$\conv(A)$ d'une partie~$A$ de~$V$ est l'intersection de
la famille des parties convexes qui contiennent~$A$;
c'est la plus petite partie convexe de~$V$ qui contient~$A$.

L'image, resp. l'image réciproque, d'une partie convexe par une application
affine est une partie convexe.

\subsection{}
Soit $C$ une partie convexe de~$V$, non vide,
et soit $W$ le sous-espace affine de~$V$ engendré par~$C$.
L'adhérence~$\overline C$ de~$C$ est 
contenue dans~$W$; elle est convexe.
(Soit en effet $x,y$ des points de~$\overline C$ et soit $t\in[0,1]$;
écrivons~$x$ et~$y$ comme limites de suites~$(x_n)$ et~$(y_n)$
de points de~$C$; alors $(1-t)x+ty=\lim (1-t)x_n+ty_n$ est limite
de points de~$C$ donc appartient à~$\overline C$.)

Si $W\neq V$, l'intérieur de~$C$ dans~$V$ est vide.
On appelle \emph{intérieur relatif} de~$C$, et on note~$\mathring C$
ou~$C^\circ$, l'intérieur de~$C$ dans~$W$.

Si $C$ est vide, on pose $\mathring C=\emptyset$.

\begin{lemm}
Soit $C$ une partie convexe d'un espace affine~$V$.
Soit $x\in\mathring C$ et $y\in\overline C$.
Alors le segment ouvert $\mathopen]x,y\mathclose[$ est contenu
dans~$\mathring C$.
\end{lemm}
\begin{proof}
Quitte à remplacer~$V$ par le sous-espace affine engendré par~$C$,
on peut supposer que $C$ engendre l'espace affine~$V$, de sorte
que $\mathring C$ est l'intérieur topologique de~$C$.
Soit $\delta$ un nombre réel~$>0$ tel que 
la boule ouverte $B(x,\delta)$ soit contenue dans~$C$, donc dans~$\mathring C$.
Soit $t\in\mathopen]0,1\mathclose[$ et soit $z=(1-t)x+ty$.
Par définition de l'adhérence de~$C$, il existe un point~$y'\in C$
dans la boule ouverte de centre~$y$ et de rayon~$t\delta/2(1-t)$;
écrivons $y'=y+v$, où $v$ est un vecteur de~$V$.
Soit $u$ un vecteur de~$V$ de norme~$<\delta/2(1-t)$ et
posons $x'=x-\frac t{1-t}v+\frac1{1-t}u$, de sorte que
\[ z+u=(1-t) (x- \frac t{1-t}v + \frac1{1-t}u) + t (y+v) = (1-t) x'+t y'. \]
Par construction, on a $x'\in C$ et $y'\in C$, donc $z+u\in C$.
Cela prouve que la boule ouverte de centre~$z$ et de rayon $\delta/2(1-t)$
est contenue dans~$C$; en particulier, $z\in\mathring C$.
\end{proof}

\begin{coro}
L'intérieur relatif d'une partie convexe est convexe.
\end{coro}

\begin{coro}\label{coro.interieur-adherence}
Soit $C$ une partie convexe d'un espace affine~$V$.
\begin{enumerate}\def\theenumi{\roman{enumi}}
\item On a $\overline C=\overline{C^\circ}$;
en particulier, l'intérieur relatif $\mathring C$ de~$C$ est dense dans~$C$.
\item On a $\mathring C=(\overline C)^\circ$;
\item Si $D$ est une partie convexe de~$V$,
les assertions $\mathring C=\mathring D$ et $\overline C=\overline D$
sont équivalentes.
\end{enumerate}
\end{coro}
\begin{proof}
On peut supposer que $C$ n'est pas vide, puis que $V$
est l'espace affine engendré par~$C$.

(i) Démontrons d'abord que $\mathring C$ n'est pas vide.
Puisque $C$ engendre affinement~$V$, il existe une suite $(x_0,\dots,x_n)$
de points de~$C$ qui est un repère affine de~$V$. Alors
$C$ contient le simplexe $\conv(x_0,\dots,x_n)$ de sommets ces points. 
Le simplexe ouvert, ensemble des combinaisons convexes
des~$x_i$ à coefficients strictement positifs, est ouvert dans~$V$,
non vide,
et contenu dans~$C$. Cela prouve que $\mathring C$ n'est pas vide.
Fixons alors un point $a\in \mathring C$.

On a évidemment  $\overline {C^\circ}\subset\overline C$.
Inversement, soit $x\in \overline C$. Le segment $\mathopen]a,x\mathclose[$
est contenu dans~$\mathring C$, et $x$ en est un point adhérent.
Cela prouve que $x$ appartient à l'adhérence de~$\mathring C$, 
ce qu'il fallait démontrer.

(ii) On a $\mathring C\subset \overline C^\circ$. Inversement,
soit $x$ un point intérieur de~$\overline C$.
Soit $S$ un simplexe de~$V$, contenu dans~$\overline C$
tel que $x\in\mathring S$, c'est-à-dire que $x$ soit barycentre
des sommets $(x_0,\dots,x_n)$ de~$S$ à coefficients strictement positifs.
L'ensemble des familles $(y_0,\dots,y_n)$ de~$V^{n+1}$
qui sont des repères affines de~$V$ par rapport auquel les coordonnées
barycentriques de~$x$ sont strictement positives
est un ouvert de~$V^{n+1}$. Il contient le point~$(x_0,\dots,x_n)$
de $\overline C^{n+1}$, donc contient un point $(y_0,\dots,y_n)$
de~$C^{n+1}$. Cela fournit un simplexe~$T$ contenu dans~$C$
tel que $x\in\mathring T$; a fortiori, $x\in\mathring C$.

(iii) Cela se déduit des assertions~(i) et~(ii).
\end{proof}

\begin{prop}\label{prop.intersection-convexe}
Soit $C$ et $D$ des parties convexes d'un espace affine~$V$.
Supposons que $\mathring C\cap \mathring D$ ne soit pas vide.
On a alors
\[ (C\cap D)^\circ = C^\circ \cap D^\circ
\qquad\text{et}\qquad 
  \overline{C\cap D}=\overline C\cap \overline D. \]
\end{prop}
\begin{proof}
Fixons un point $a\in\mathring C\cap \mathring D$.

Commençons par démontrer les égalités
\[ \overline{C^\circ \cap D^\circ }=\overline{C\cap D}=\overline C\cap \overline D.\]
Les inclusions 
\[ \overline{C^\circ\cap D^\circ}\subset
\overline {C\cap D} \subset  \overline C\cap \overline D \]
sont évidentes, car $\overline C$ et $\overline D$ sont des parties fermées
de~$V$ contenant $C\cap D$, et $C\cap D$ contient $C^\circ\cap D^\circ$.
Inversement, soit $x\in\overline C\cap \overline D$.
Le segment~$\mathopen]a,x\mathclose[$ étant contenu dans~$C^\circ$
et dans $D^\circ$, il est contenu dans  $C^\circ\cap D^\circ$.
Comme le point~$x$ y est adhérent, on a $x\in\overline{C^\circ\cap D^\circ}$.

Puisque $C^\circ\cap D^\circ$ est une partie convexe
de~$V$, de même adhérence que $C\cap D$, elle a même intérieur relatif.
On a donc $(C^\circ \cap D^\circ)^\circ=(C\cap D)^\circ$.
En particulier, $(C\cap D)^\circ\subset C^\circ\cap D^\circ$.

Soit $x$ un point de~$\mathring C\cap \mathring D$.
Cet ensemble contient un voisinage de~$x$ dans l'intersection
des espaces affines engendrés par~$C$ et~$D$; 
en particulier, il contient un voisinage de~$x$  dans
le sous-espace affine engendré par~$C\cap D$, 
ce qui prouve que $\mathring C\cap \mathring D$ est contenu
dans $(C\cap D)^\circ$.
\end{proof}

\begin{prop}\label{prop.image-convexe}
Soit $f\colon V\to W$ une application affine.

\begin{enumerate}
\item
Soit $C$ une partie convexe de~$V$.
Alors $f(C)$ est une partie convexe de~$W$ et l'on a $f(C)^\circ=f(C^\circ)$.
\item
Soit $D$ une partie convexe de~$W$. Alors, $f^{-1}(D)$ est une partie
convexe de~$V$. Si, de plus, $f^{-1}(D^\circ)\neq \emptyset$,
alors $f^{-1}(D)^\circ=f^{-1}(D^\circ)$.
\end{enumerate}
\end{prop}
\begin{proof}
Pour la première assertion, on peut supposer que $C$
engendre~$V$ et que $f$ est surjective.
Comme $f$ est continue et $C^\circ$ est dense dans~$\overline C$, 
on a 
\[ f(C)\subset f(\overline C)= f(\overline{C^\circ})
\subset \overline{f(C^\circ)} \subset \overline{f(C)}. \]
Par suite, $f(C^\circ)$ et $f(C)$ ont même adhérence,
donc même intérieur relatif. En particulier,
$f(C)^\circ=f(C^\circ)^\circ\subset f(C^\circ)$.
Puisque $f$ est une application ouverte et que $C^\circ$ est 
ouvert dans~$V$, $f(C^\circ)$ est ouvert dans~$W$;
on a donc $f(C^\circ)\subset f(C)^\circ$.
L'égalité $f(C)^\circ=f(C)^\circ$ en résulte.

Pour démontrer la seconde assertion, supposons d'abord que
$f$ soit l'inclusion d'un sous-espace affine.
Alors, $f^{-1}(D)=D\cap V$, et $f^{-1}(D^\circ)=D^\circ \cap V
= (D\cap V)^\circ=f^{-1}(D)^\circ$ en vertu de la proposition précédente.
Dans le cas général, choisissons  un sous-espace affine maximal~$V_1$
de~$V$ tel que $f|_{V_1}$ soit injectif. Cela permet de supposer
que $V=V_1\times V_2$ et que $f$ est de la forme $(v_1,v_2)\mapsto f_1(v_1)$,
où $f_1\colon V_1\to W$ est injectif. Alors
\[ f^{-1}(D^\circ)=f_1^{-1}(D^\circ)\times V_2
 = f_1^{-1}(D)^\circ \times V_2=(f_1^{-1}(D)\times V_2)^\circ
 = f^{-1}(D)^\circ, \]
ainsi qu'il fallait démontrer.
\end{proof}

\chapter{Structure des variétés abéloïdes}

\section{Géométrie analytique et géométrie formelle}


Soit $k$ un corps valué complet; 
on suppose que sa valuation n'est pas triviale.

\subsection{}\label{ss.formel-spec}
Un modèle formel~$\mathfrak X$ d'un espace strictement $k$-analytique~$X$
est un $k^\circ$-schéma formel \emph{plat}, localement topologiquement
de présentation finie  muni d'un isomorphisme de~$\mathfrak X_\eta $
sur~$X$ (qu'on considérera être une égalité).
Nous renvoyons par exemple au~\S1 de~\cite{berkovich1994} pour la construction
(dans un cadre un peu plus général où~$\mathfrak X$ n'est pas supposé plat).

On dispose aussi d'une application de spécialisation
$\pi\colon \mathfrak X_\eta\to \mathfrak X_s$; elle est anticontinue
et surjective.
La preuve de cette surjectivité ne semble 
pas explicitée dans la littérature;
pour les points fermés, renvoyons à la proposition~8.3/8 de~\cite{bosch2014};
on s'y ramène par changement de base.

\subsection{}
Le théorème suivant est une variante du théorème de Raynaud ;
sa preuve consiste essentiellement à vérifier que la construction
faite par Raynaud fournit un résultat un peu plus précis
que celui traditionnellement donné.
\begin{theo}\label{theo.raynaud}
Soit $X$ un espace strictement $k$-analytique topologiquement séparé.
Soit $(V_i)_{i\in I}$ un G-recouvrement localement fini de~$X$
par des domaines strictement analytiques compacts,
pour tout~$i\in I$, soit $\mathfrak V_i$ un modèle formel de~$V_i$.
Il existe un modèle formel $\mathfrak X$ de~$X$ et
un recouvrement ouvert $(\mathfrak U_i)_{i\in I}$ de~$\mathfrak X$
tels que pour tout~$i\in I$, $\mathfrak U_{i,\eta}=U_i$
et $\mathfrak U_i$ est un éclatement admissible de~$\mathfrak V_i$.
\end{theo}

\begin{lemm}\label{lemm.raynaud-fini}
Le théorème~\ref{theo.raynaud} est vrai lorsque l'ensemble~$I$ est fini.
\end{lemm}
\begin{proof}
Soit $X$ un espace strictement $k$-analytique,
réunion de deux domaines strictement analytiques compacts~$U$ et~$V$,
soit $\mathfrak U$ et $\mathfrak V$ des modèles formels de~$U$ et~$V$.
Démontrons qu'il existe alors un modèle formel~$\mathfrak X$ de~$X$ et
deux ouverts $\mathfrak U'$ et $\mathfrak V'$ de~$\mathfrak X$
tels que:
\begin{enumerate}
\item $\mathfrak U'_\eta=U$ et $\mathfrak U'$ est un éclatement
admissible de~$\mathfrak U$;
\item $\mathfrak V'_\eta=V$ et $\mathfrak V'$ est un éclatement
admissible de~$\mathfrak V$.
\end{enumerate}
Posons $W=U\cap V$; 
c'est un domaine compact de~$X$, car $X$ est topologiquement séparé.
Le lemme~8.4/5 de~\cite{bosch2014} appliqué
à~$\mathfrak U$ et au domaine analytique $W$ de~$U=\mathfrak U_\eta$
fournit 
un éclatement admissible $\mathfrak U'\to \mathfrak U$
et un ouvert formel~$\mathfrak W_1'$ de~$\mathfrak U'$
de fibre générique $W$.
De même, il existe un éclatement admissible $\mathfrak V'\to \mathfrak V$
et un ouvert formel~$\mathfrak W_2'$ de~$\mathfrak V'$
de fibre générique $W=V\cap U$.
D'après le lemme~8.4/4 (d) de~\cite{bosch2014},
il existe un diagramme d'éclatements admissibles 
\[ \begin{tikzcd}[column sep=small]
   & \mathfrak W'' \ar{dl} \ar{dr} \\
  \mathfrak W'_1 &  & \mathfrak W'_2 \end{tikzcd}
\]
induisant l'identité de~$W $ sur les fibres génériques.
Comme $\mathfrak W'_1$ et $\mathfrak W'_2$ sont quasi-compacts,
ces deux éclatements admissibles  se prolongent en des éclatements admissibles
$\mathfrak U'\leftarrow \mathfrak U''$ et $\mathfrak V''\to \mathfrak V'$
(\cite{bosch2014}, proposition~8.2/13).
Soit $\mathfrak X$ le recollement de $\mathfrak U''$ et $\mathfrak V''$
le long de $\mathfrak W''$. C'est un modèle formel de~$X$.
Comme le composé de deux éclatements admissibles est un éclatement
admissible (\cite{bosch2014}, lemme~8.2/11),
le recouvrement ouvert $(\mathfrak U'',\mathfrak U'')$ satisfait
les conditions requises.
\end{proof}

\begin{lemm}\label{lemm.raynaud-N}
Le théorème~\ref{theo.raynaud} est vrai lorsque
l'ensemble $I=\N$ et que $V_m\cap V_n=\emptyset$
dès que $\abs{m-n}\geq 2$.
\end{lemm}
\begin{proof}
Pour tout~$n$, posons $W_n=V_0\cup\dots\cup V_n$.
C'est un espace strictement analytique compact
dont $(V_0,\dots,V_n)$ est un G-recouvrement fini par des domaines
strictement analytiques;
nous allons en construire un modèle $\mathfrak W_n$
muni d'un recouvrement ouvert $(\mathfrak V_{n,i})_{i\leq n}$
tel que $(\mathfrak V_{n,i})_\eta=V_i$ pour tout~$i$,
de sorte que les propriétés suivantes soient satisfaites:
\begin{enumerate}
\item Pour tous~$i,n$ tels que $i\leq n$,
$\mathfrak V_{n,i}$ domine~$\mathfrak V_i$;
\item Pour tous $i,m,n$ tels que $i\leq m\leq n$, $\mathfrak V_{n,i}$ domine
$\mathfrak V_{m,i}$;
\item Pour tous $i,m,n$ tels que $i\leq m-2$ et $m\leq n$, 
alors $\mathfrak V_{n,i}=\mathfrak V_{m,i}$.
\end{enumerate}

Pour $n=0$, on pose $\mathfrak W_{0}=\mathfrak V_{0,0}=\mathfrak V_{0}$.

Supposons $\mathfrak W_m$ et $\mathfrak V_{m,i}$ 
construits pour $i\leq m\leq n$; construisons $\mathscr W_{n+1}$.

Commençons par recoller $\mathfrak V_{n,n}$ et $\mathfrak V_{n+1}$.
D'après le lemme~\ref{lemm.raynaud-fini},
appliqué à $V_n\cup V_{n+1}$ et à son recouvrement $(V_n,V_{n+1})$
et à la famille de modèles $(\mathfrak V_{n,n},\mathfrak V_{n+1})$
il existe un schéma formel~$\mathfrak S$, un recouvrement ouvert
$(\mathfrak V'_{n,n},\mathfrak V'_{n+1})$ de~$\mathfrak S$
induisant le recouvrement $(V_n,V_{n+1})$,
où $\mathfrak V'_{n,n}$ est un éclatement admissible de $\mathfrak V_{n,n}$
et $\mathfrak V'_{n+1}$ est un éclatement admissible de~$\mathfrak V_{n+1}$.
Soit $I$ le faisceau cohérent d'idéaux ouverts sur~$\mathfrak V_{n,n}$
définissant~$\mathfrak V'_{n,n}$.
Pour $i\leq n-2$, la fibre générique~$V_n$ de~$\mathfrak  V_{n,n}$ 
est disjointe de celle, $V_i$, de~$\mathfrak V_{n,i}$.
D'après la proposition~8.2/13 de~\cite{bosch2014}, 
il existe un faisceau cohérent d'idéaux ouverts~$J$ sur~$\mathscr W_n$
prolongeant~$I$ et dont la restriction à l'union des~$\mathscr V_{n,i}$
pour $i\leq n-2$ est le faisceau structural.
Soit $\mathfrak W'_n$ l'éclatement admissible de~$\mathfrak W_n$
le long de~$J$; c'est un modèle de~$W_n$ dont $\mathfrak V'_{n,n}$
est un ouvert.

Soit $\mathfrak W_{n+1}$ le schéma formel obtenu par
recollement de~$\mathfrak W'_n$ et~$\mathfrak S$
le long de $\mathfrak W'_{n,n}$.
Par construction, $\mathfrak V'_{n+1}$ en est un ouvert que l'on
note $\mathfrak V_{n+1,n+1}$.
On pose aussi $\mathfrak V_{n+1,n}=\mathfrak V'_{n,n}$.
Enfin, pour $i\leq n-1$, on note $\mathfrak V_{n+1,i}$
l'image réciproque de $\mathfrak V_{n,i}$ dans
l'ouvert~$\mathfrak W'_n$ de~$\mathfrak W_{n+1}$.
Si $i\leq n-2$, $\mathfrak V_{n+1,i}=\mathfrak V_{n,i}$.

La construction est ainsi obtenue par récurrence.

Pour tout~$i$, notons~$\mathfrak U_i$ la limite 
du système inductif essentiellement constant $(\mathfrak V_{n,i})_{n\geq i}$.
La famille $(\mathfrak U_i)$ se recolle en un schéma formel~$\mathfrak X$
qui satisfait les conditions requises.
\end{proof}

\begin{proof}[Preuve du théorème~\ref{theo.raynaud}]
On peut supposer que $X$ est connexe. 

Définissons une suite $(W_n)$ de domaines strictement analytiques compacts
de~$X$ et une suite $(I_n)$ de parties finies de~$I$
telles que $W_n=\bigcup_{i\in I_n} V_i$.
Choisissons un élément $i_0\in I$ et posons $W_0=V_{i_0}$; posons $I_0=\{i_0\}$.
Si $W_0,\dots,W_n$, et $I_0,\dots,I_n$ sont définis, 
posons $I_{n+1}$ l'ensemble des $i\in I$ tels que 
tels que $V_i$ rencontre~$W_n$, mais $V_i$ ne rencontre
aucun des~$W_m$, pour $m<n$, et soit $W_{n+1}$ la réunion des~$V_i$,
pour $i\in I_{n+1}$.
Comme le recouvrement~$(V_i)$ est localement fini, 
l'ensemble $I_{n+1}$ est fini, de sorte que $W_{n+1}$ est un domaine
strictement analytique compact.

Pour tout~$n$, choisissons un modèle formel~$\mathfrak W_n$
de~$W_n$ possédant un recouvrement ouvert $(\mathfrak V_{n,i})_{i\in I_n}$,
où $\mathfrak V_{n,i}$ est, pour tout $i\in I_n$, un éclatement
admissible de~$\mathfrak V_i$.

Appliquons le lemme~\ref{lemm.raynaud-N}
à l'espace~$X$ et à la famille $(\mathfrak W_n)$.
On en déduit l'existence d'un modèle $\mathfrak X$ muni
d'un recouvrement ouvert $(\mathfrak W'_n)$ où, pour tout~$n$,
$\mathfrak W'_n$ est un éclatement admissible de~$\mathfrak W_n$.
Pour tout~$i\in I$, il existe un unique entier~$n$
tel que $i\in I_n$; 
notons $\mathfrak U_i$ l'image réciproque
de~$\mathfrak V_{n,i}$ dans~$\mathfrak W'_n$.
Le modèle~$\mathfrak X$ et la famille $(\mathfrak U_i)$ satisfont
les conditions du théorème~\ref{theo.raynaud}. 
\end{proof}

\begin{coro}\label{coro.raynaud}
Tout espace strictement $k$-analytique paracompact possède un modèle
formel.
\end{coro}
\begin{proof}
Soit $X$ un tel espace. Il est topologiquement séparé,
par définition d'un espace topologique paracompact.
Par hypothèse, $X$ possède un G-recouvrement strictement affinoïde.
Puisque $X$ est paracompact, ce G-recouvrement peut être raffiné
en un G-recouvrement~$(X_i)_{i\in I}$ qui est localement fini.
Pour tout~$i$, choisissons un modèle de l'espace affinoïde~$X_i$
(choisir un système d'équations à coefficients entiers
et quotienter par la $k^{\circ\circ}$-torsion).
D'après le théorème~\ref{theo.raynaud}, il existe
un modèle formel~$\mathfrak X$ de~$X$.
\end{proof}

\begin{coro}\label{coro.modele-fonctions}
Soit $X$ un espace strictement analytique paracompact,
soit $(f_\alpha)_{\alpha\in A}$ 
une famille finie d'éléments de~$\mathscr O_X(X)$
telles que $\abs{f_\alpha(x)}\leq 1$ pour tout $\alpha\in A$ et tout $x\in X$.
Il existe un modèle $\mathfrak X'$ de~$\mathfrak X$
tel que chaque~$f_\alpha$ appartienne à~$\mathscr O_{\mathfrak X'}(\mathfrak X')$.
\end{coro}
\begin{proof}
Soit $(\mathfrak X_i)$ un recouvrement ouvert affine de~$\mathfrak X$
dont les fibres génériques~$(X_i)$ 
forment un G-recouvrement localement fini de~$X$.
D'après le lemme~8.4/6 de \cite{bosch2014},
il existe pour tout~$i$ un éclatement admissible $\mathfrak X'_i$
de~$\mathfrak X_i$ tel que $f_\alpha|_{X_i}$ appartient
à $\mathscr O(\mathfrak X'_i)$  pour tout $\alpha\in A$.
D'après le théorème~\ref{theo.raynaud},
il existe un modèle~$\mathfrak X'$ de~$X$ 
possédant un recouvrement ouvert~$(\mathfrak X''_i)$
tel que $\mathfrak X''_i$ domine~$\mathfrak X'_i$ pour tout~$i$. Il convient.
\end{proof}

\begin{theo}\label{theo.app-faisceau}
Soit $X$ un espace strictement $k$-analytique paracompact
et soit~$\mathfrak X$ un modèle formel de~$X$.
Soit $\mathscr M$ un faisceau cohérent sur~$X_\groth$.

\begin{enumerate}
\item
Il un faisceau cohérent~$\mathfrak M$
de~$\mathscr M$ sur~$\mathfrak X$
dont la fibre générique est égale à~$\mathscr M$.

\item 
Si $\mathscr M$ est localement libre et $(\mathfrak X,\mathfrak M)$
en est un modèle, il existe un éclatement admissible
$p\colon \mathfrak X'\to\mathfrak X$ tel que le transformé
strict de~$\mathfrak M$ par~$p$ soit localement libre.
\end{enumerate}
\end{theo}
Apparemment, la preuve ne se trouve dans la littérature
(proposition~5.6 de~\cite{bosch-lutkebohmert1993}, par exemple) 
que sous l'hypothèse que $X$ est compact.
Comme nous allons voir, la preuve donnée s'applique en général.
\begin{proof}
\begin{enumerate}
\item
On peut supposer que $X$ est connexe;
puisqu'il est paracompact, 
$\mathfrak X$ est la réunion d'une \emph{suite} $(\mathfrak U_n)_{n\in\N}$
d'ouverts formels affines dont les fibres génériques
forment un G-recouvrement affinoïde de~$X$.
Pour tout entier~$n$, 
posons aussi $\mathfrak V_n=\bigcup_{m=0}^n \mathfrak U_m$.

Soit $n$ un entier.
La fibre générique~$\mathfrak U_{n,\eta}$ de~$\mathfrak U_n$ étant
affinoïde,
il existe un faisceau cohérent 
$\mathfrak M_n$ sur $\mathfrak U_n$, sans $k^\circ$-torsion,
dont la fibre générique est égale à~$\mathscr M$
sur $\mathfrak U_{n,\eta}$; il suffit par exemple
de prendre une présentation de~$\mathscr M$ à coefficients entiers
et de quotienter par la $k^{\circ\circ}$-torsion.

On construit alors par récurrence un faisceau cohérent~$\mathfrak M'_n$
sur~$\mathfrak V_n$ dont la restriction à~$\mathfrak V_{n-1}$
est égale à $\mathfrak M'_{n-1}$ et dont la fibre générique
est $\mathscr M$ sur $\mathfrak U_{n,\eta}$. Le point-clé est
le lemme~5.7 de~\cite{bosch-lutkebohmert1993} qui explique comment,
étant donné $\mathfrak M'_n$, comment modifier $\mathfrak M_{n+1}$
de sorte qu'il se recolle avec $\mathfrak M'_n$ sur $\mathfrak V_n\cap \mathfrak U_{n+1}$.\footnote{Avec les notations de \loccit,
il convient plutôt de recoller $\mathscr M_1$ et $g^{-n}\mathscr M'_2+g^m\mathscr M_2$, où $m$ est un entier tel que $g^m\mathfrak M_{21}\subset\mathscr M_{12}$.}
Cela fournit un modèle formel $\mathfrak M$ de~$\mathscr M$.

\item
Supposons maintenant que $\mathscr M$ est localement libre et
expliquons comment utiliser le théorème de platification par éclatements
de Raynaud
(\cite{bosch-lutkebohmert1993b}, théorème~4.1)
pour transformer un modèle formel~$\mathfrak M$ de~$\mathscr M$
en un modèle localement libre.
On se ramène au cas où $X$ est connexe;
reprenons le recouvrement ouvert croissant 
$(\mathfrak V_n)$ considéré précédement.
On va construire un système projectif $(\mathfrak X_{n})$
de modèles formels de~$X$, avec $\mathfrak X_{-1}=\mathfrak X$,
dont les morphismes de transition sont des éclatements admissibles,
de sorte qu'au-dessus de~$\mathfrak V_n$,
le transformé strict de~$\mathfrak M$ sur~$\mathfrak X_n$
soit localement libre et 
le morphisme de transition $\mathfrak X_{n+1}\to\mathfrak X_n$ 
soit un isomorphisme.
La limite projective de ce système est un schéma formel $\mathfrak X'$
qui s'identifie à~$\mathfrak X_n$ au-dessus de~$\mathfrak V_n$,
le morphisme canonique $\mathfrak X'\to\mathfrak X$ est un éclatement
admissible et le transformé strict de~$\mathfrak M$ sur~$\mathfrak X'$
est localement libre; 
c'est un modèle de~$\mathscr M$.

On pose $\mathfrak X_{-1}=\mathfrak X$.
Supposons $\mathfrak X_{n}$ construit, 
notons $\mathfrak F$ le transformé strict de~$\mathfrak M$ sur~$\mathfrak X_{n}$, notons $\mathfrak W$ l'image réciproque de~$\mathfrak V_n$
et $\mathfrak W'$ celle de $\mathfrak V_{n+1}$.
Il existe un idéal cohérent ouvert~$\mathscr J'_n$
de $\mathscr O_{\mathfrak W'}$ tel que le transformé
strict de~$\mathfrak F|_{\mathfrak W'}$ par l'éclatement admissible
correspondant soit localement libre;
de plus, on peut suppose que $\mathscr J'_n$ est trivial 
sur l'ouvert~$\mathfrak W$ sur lequel $\mathfrak F$ est localement libre.
Prolongeons l'idéal $\mathscr J_n$ en un idéal cohérent ouvert
de~$\mathscr O_{\mathfrak X_n}$ (\cite{bosch2014}, proposition~8.2/13)
et prenons pour espace~$\mathfrak X_{n+1}$ l'éclatement admissible
correspondant. Le système $(\mathfrak X_n)$ ainsi défini convient.
\qedhere
\end{enumerate}
\end{proof}

%
%
%

\section{Adhérence de Zariski des diviseurs de Cartier}

\begin{prop}\label{prop.adherence-cartier}
Soit $A$ un anneau de valuation, soit $K$ son corps des fractions.
Soit $f\colon X\to\Spec(A)$ un morphisme lisse de schémas.
Soit $Y$ un diviseur de Cartier effectif de~$X_K$.
L'adhérence schématique de~$Y$ dans~$X$ est un diviseur de Cartier
effectif de~$X$.
\end{prop}
(Voir aussi~\cite{Lutkebohmert-1990}, lemme~2.9.)
\begin{proof}
On peut supposer que $X$ est de type fini et de dimension relative constante.
Notons $Z$ l'adhérence schématique de~$Y$ dans~$X$.
Par définition de l'adhérence schématique,
$\mathscr O_Z$ est sans-$A$-torsion,  donc $Z$ est plat sur~$A$.
D'après le théorème~3 de~\cite{Nagata-1966}, il est alors de présentation finie.
En particulier, il est de codimension relative~$1$ en tout point
et ne contient donc aucune composante irréductible d'une fibre
de~$X\to\Spec(A)$.
Les anneaux locaux de~$X$ sont intègres.
Il s'agit donc de prouver  que l'idéal de~$Z$ est principal  
au voisinage de tout point.

Par approximation noethérienne, il existe une sous-$\Z$-algèbre
normale et de type fini~$\mathscr A$  de~$A$, 
un morphisme lisse $\phi\colon\mathscr X\to\Spec(\mathscr A)$
et un sous-schéma fermé~$\mathscr Z$ de~$\mathscr X$
induisant $f\colon X\to\Spec(A)$ et~$Z$ par changement de base à~$A$.

Soit $z\in Z$, posons $s=f(z)$;
soit $\zeta\in\mathscr Z$ 
et $\sigma=\phi(\zeta)\in \Spec(\mathscr A)$ leurs images.
Par hypothèse, aucune composante irréductible de $f^{-1}(s)$ n'est contenue
dans~$Z$; par conséquent, aucune composante irréductible de~$\phi^{-1}(\sigma)$
n'est contenue dans~$\mathscr Z$.
D'après la proposition~21.14.3 de~\cite{ega4.4}
(théorème de Ramananujam--Samuel), 
l'idéal de~$\mathscr Z$ dans~$\mathscr X$ est principal au voisinage de~$\zeta$.
Par suite, l'idéal de~$Z$ dans~$X$ est principal au voisinage de~$z$.
Cela conclut la démonstration.
\end{proof}

\begin{coro}\label{coro.extension-pic}
Soit $A$ un anneau de valuation, soit $K$ son corps de fractions.
Soit $f\colon X\to\Spec(A)$ un morphisme lisse et de présentation finie
de schémas.
L'homomorphisme de $\Pic(X)$ dans~$\Pic(X_K)$ est surjectif.

Si les fibres de~$f$ sont géométriquement intègres, 
cet homomorphisme est bijectif.
\end{coro}
(Pour l'injectivité, on utilisera seulement de~$A$ que c'est un anneau
normal de groupe de Picard trivial.)
\begin{proof}
Soit $L$ un fibré en droites sur $X_K$. 
Dans chaque composante connexe de~$X_K$, on choisit un  ouvert non vide 
sur lequel $L$ est trivialisable et on en choisit une trivialisation;
la famille de ces trivialisations définit une section rationnelle inversible~$s$
de~$L$.
Comme $X_K$ est lisse, et en particulier localement factoriel, 
chaque composante du diviseur de Cartier~$\div(s)$  est elle-même
un diviseur de Cartier, ce qui permet d'écrire~$\div(s)$
comme la différence $D-D'$ de deux diviseurs de Cartier \emph{effectifs} 
sur~$X_K$.
D'après la proposition~\ref{prop.adherence-cartier}, leurs
adhérences schématiques~$E$ et $E'$ sont des diviseurs
de Cartier sur~$X$. Par construction $\mathscr O_X(E-E')|_{X_K}$
est isomorphe à~$\mathscr O_{X_K}(D-D')$, donc à~$L$.
Cela démontre la surjectivité de l'homomorphisme indiqué.

Supposons que les fibres de~$f$ sont géométriquement intègres
et démontrons l'injectivité de cet homomorphisme.
Soit $L$ un fibré en droites sur~$X$ dont
la restriction à~$X_K$ est trivialisable.
Comme la propriété pour les fibres d'un morphismes d'être géométriquement
intègres est constructible, il existe un sous-anneau~$R$ de~$A$,
qui est une $\Z$-algèbre de type fini normale, 
un morphisme $\phi\colon\mathscr X\to\Spec(R)$, 
lisse et de présentation finie, à fibres géométriquement intègres,
et un fibré en droites~$\mathscr L$ sur~$\mathscr X$
tel que $\mathscr L_{\Frac(R)}$ soit trivialisable,
induisant~$f$ et~$L$ par changement de base.
Compte-tenu de~\cite[Err., 53, (21.4.13)]{ega4.4},
il existe un fibré en droites~$\mathscr M$ sur~$\Spec(R)$
tel que $\mathscr L$ est isomorphe à~$\phi^*\mathscr M$.
Il s'ensuit que $L$ est isomorphe à~$f^*M$, où $M$
est le fibré en droites sur~$\Spec(A)$ déduit de~$\mathscr M$.
Puisque $A$ est local, $M$ est trivialisable, et $L$ l'est donc aussi.
\end{proof}

\section{Corps de définition des domaines analytiques}

\begin{prop}\label{prop.descente-1}
Soit $X$ un espace $k$-analytique topologiquement séparé.
Soit~$K$ une extension complète de~$k$ et 
soit~$W$ un domaine analytique de~$X_K$.

\begin{enumerate}
\item Il existe au plus un domaine analytique compact~$V$ de~$X$
tel que $W=V_K$.
\item Soit $K_0$ un sous-corps dense de~$K$ contenant~$k$.
Si $W$ est compact, il existe un sous-corps~$k'$ de~$K_0$,
de type fini sur~$k$, dont on note~$K'$ l'adhérence dans~$K$,
et un domaine analytique~$V'$ de~$X_{K'}$ tels que $W=V'_K$.
\end{enumerate}
\end{prop}
\begin{proof}
Notons $\pi\colon X_K\to X$  de changement de base; il est surjectif.
Si $V$ est un domaine analytique de~$X_K$ tel que $W=V_K$,
on a $V=\pi(W)$ et $W=\pi^{-1}(V)$, d'où l'unicité. 

Cela démontre plus généralement
qu'il existe un domaine analytique~$V$ de~$X$ tel que $W=V_K$
si et seulement si $\pi(W)$ est un domaine analytique 
de~$X$ tel que $W=\pi^{-1}(\pi(W))$.
On constate ainsi que cette propriété  est locale pour la G-topologie de~$X$.

Comme $\pi(W)$ est compact, il existe une famille finie $(U_i)$
de domaines affinoïdes de~$X$ dont la réunion recouvre~$\pi(W)$.
Il suffit de prouver le résultat pour les domaines analytiques $W\cap \pi^{-1}(U_i)$ de $U_{i,K}$, quitte à prendre, 
sur l'extension composée des extensions de type fini qui apparaissent,
la réunion des domaines analytiques construits.
On suppose donc que $X$ est affinoïde. Dans ce cas, $X_K$ est affinoïde,
et $W$ est réunion finie de domaines rationnels de~$X_K$,
en vertu du théorème de Gerritzen--Grauert. 

On peut donc supposer
que $W$ est un domaine rationnel de~$X_K$. Il est ainsi décrit
par une combinaison d'inégalités $\abs{f_1}\leq \lambda_1 \abs g$,
\dots, $\abs{f_n}\leq \lambda_n \abs g$, où $f_1,\dots,f_n,g$
sont des éléments de~$\mathscr O(X_K)$ sans zéro commun
et $\lambda_1,\dots,\lambda_n$ des nombres réels strictement positifs.
En particulier, $g$ ne s'annule pas sur~$W$; comme $W$ est compact,
$g$ est minorée sur~$W$ et on peut ajouter à ce système d'inégalités
une inégalité de la forme $1\leq \lambda_0 \abs g$.
L'anneau $\mathscr O(X)\otimes_k K_0$ est dense dans~$\mathscr O(X_K)$.
Grâce à l'inégalité ultramétrique, on peut remplacer les~$f_i$ et~$g$
par des éléments de $\mathscr O(X)\otimes_k K_0$
sans changer le domaine décrit par les inégalités ci-dessus.
Il suffit de prendre le sous-corps~$k'$ de~$K_0$ engendré par les 
coefficients d'une représentation des~$f_i$ et~$g$ dans
le produit tensoriel $\mathscr O(X)\otimes_K K_0$.
\end{proof}

\begin{coro}\label{coro.descente-1}
Supposons que $W$ est compact 
et que $K$ est la complétion d'une clôture algébrique de~$k$.

\begin{enumerate}
\item Il existe une extension finie~$k'$ de~$k$ contenue dans~$K$
et un domaine analytique compact~$V'$ de~$X_{k'}$
tels que $W=V'_K$;
\item Si $k$ n'est pas trivialement valué,
on peut choisir $k'$ séparable sur~$k$.
\end{enumerate}
\end{coro}
\begin{proof}
La première assertion est immédiate, et la seconde résulte de
ce que la clôture séparable de~$k$ dans~$K$ est dense 
dans~$K$ (\cite{bosch-g-r1984}, proposition 3.4.1/6).
\end{proof}

\begin{rema}
Supposons que la valuation de~$k$ n'est pas triviale.
La démonstration de la proposition
prouve que si $X$ est un espace strictement $k$-analytique,
et si $W$ est un domaine strictement $K$-analytique de~$X_K$,
il existe une extension~$k'$ comme dans la proposition
(ou dans le corollaire)
et un domaine strictement~$k'$-analytique~$V'$ de~$X_{k'}$
tels que $W=X_K$. On procède de la même manière, en prenant $\lambda_1,\dots,\lambda_n$ égaux à~$1$ et $\lambda_0$ dans $\abs{k^\times}$.
\end{rema}

\section{Espaces formellement analytiques}

Soit $k$ un corps ultramétrique complet dont la valeur absolue est non triviale.
Dans ce paragraphe, tous les espaces analytiques considérés seront
et nous noterons $\widetilde\cdot$ les réductions classiques, non graduées.

\begin{defi}
Soit $X$ un espace strictement affinoïde
et soit $\pi\colon X\to\widetilde X$ l'application de réduction.
On dit qu'un domaine analytique~$V$ de~$X$ 
est un \emph{domaine formel} si $\pi(V)$ est un ouvert
de~$\widetilde X$ et si $V=\pi^{-1}(\pi(V))$.
\end{defi}

Cela revient à demander qu'il existe une famille finie $(f_i)$ 
de fonctions analytiques sur~$X$ de normes spectrales inférieures
ou égales à~$1$ 
telles que $V$ soit la réunion des domaines $\{ \abs{f_i(x)}=1\}$.

Un domaine formel est un domaine strictement analytique compact.

\begin{lemm}
Soit $X$ un espace strictement affinoïde 
et soit $V$ un domaine affinoïde de~$X$.
On a les équivalences:
\begin{enumerate}
\item Le domaine analytique~$V$ est un domaine formel de~$X$;
\item L'immersion de~$V$ dans~$X$ induit une immersion ouvertes
$\widetilde V\to \widetilde X$ entre leurs réductions canoniques;
\item Le domaine~$V$ est l'image réciproque d'un ouvert affine de~$\widetilde X$
par l'application de réduction.
\end{enumerate}
\end{lemm}
Notons $\pi\colon X\to \widetilde X$ l'application de réduction canonique.
L'image de~$\widetilde V$ dans~$\widetilde X$ par
le morphisme induit par l'immersion $V\hookrightarrow X$ 
est égale à~$\pi(V)$, par construction.
Ainsi, si les assertions du lemme sont satisfaites, 
la réduction canonique~$\widetilde V$ de~$V$
s'identifie à l'ouvert $\pi(V)$ de~$\widetilde X$.

Si $V$ est un domaine formel quelconque  de~$X$, 
il n'y aura donc pas d'ambiguïté à noter~$\widetilde V$
l'ouvert~$\pi(V)$ de~$\widetilde X$.
\begin{proof}
L'équivalence (i)$\Leftrightarrow$(ii) est le théorème~1.2
de~\cite{bosch1977}. Si elles sont satisfaites, $\widetilde V$ 
est un schéma affine, car $V$ est affinoïde, d'où l'assertion~(iii). 
Enfin, l'implication (iii)$\Rightarrow$(i) est évidente.
\end{proof}

\begin{defi}
Soit $X$ un espace strictement $k$-analytique.
Un \emph{atlas affinoïde formel} sur~$X$ est un G-recouvrement~$\mathscr V$
de~$X$ par des domaines strictement affinoïdes tel que pour tout couple~$(V,W)$ d'éléments de~$\mathscr V$, 
le domaine analytique $V\cap W$ est un domaine formel de~$V$ et de~$W$.
\end{defi}

Deux atlas affinoïdes formels sont dits compatibles si leur réunion
est encore un atlas affinoïde formel. C'est une relation d'équivalence.

Un espace formellement analytique est un espace strictement  $k$-analytique
muni d'une classe d'équivalence  d'atlas affinoïdes formels.
Il possède alors un unique atlas affinoïde formel maximal;
les éléments de cet atlas sont appelés
\emph{domaines affinoïdes formels}.

Un espace formellement analytique est topologiquement séparé.

\subsection{}
Soit $X$ un espace formellement analytique  et soit $\mathscr V$
son atlas affinoïde formel maximal. Pour $U\in\mathscr V$,
on note $\pi_U\colon U\to\widetilde U$ l'application de réduction canonique.

Pour tous $U,V,W\in\mathscr V$ tels que $W\subset U\cap V$,
l'application canonique $\widetilde W \to \pi_U(W)$ est un isomorphisme,
de même que l'application canonique $\widetilde W\to \pi_V(W)$.
On peut donc recoller les $\tilde k$-schémas~$\widetilde U$ et les applications
$\pi_U\colon U \to\widetilde U$, pour $U$ parcourant~$\mathscr V$,
en un $\tilde k$-schéma $\widetilde X$ 
et une application 
anticontinue surjective $\pi\colon X\to\widetilde X$.

On dit que $\widetilde X$ est la \emph{réduction canonique} de
l'espace formellement analytique~$X$.  (Elle dépend bien sûr
de l'atlas affinoïde formel choisi, et pas seulement de la structure
d'espace analytique sous-jacent.)

Comme les $\widetilde U$ sont réduits et de type fini,
le $\tilde k$-schéma $\widetilde X$ est réduit et localement de type fini.
Si $X$ est compact, il est de type fini.

\begin{exem}
\begin{enumerate}
\item
Soit $X$ un espace affinoïde. L'ensemble~$\{X\}$ est un atlas affinoïde formel sur~$X$ qui le munit d'une structure d'espace analytique formel.

\item
Soit $\mathfrak X$ un schéma formel plat 
et localement topologiquement de type fini sur~$k^\circ$.
Soit $\mathscr U$ un recouvrement de~$\mathfrak X$
par des ouverts formels affines.
Alors, la famille des $\mathfrak U_\eta$, pour $\mathfrak U\in\mathscr U$,
est un atlas affinoïde formel de~$\mathfrak X_\eta$.
La structure d'espace analytique formelle obtenue sur~$X=\mathfrak X_\eta$
ne dépend pas du choix du recouvrement~$\mathscr U$.

Le morphisme de réduction $\rho\colon X\to\mathfrak X_s$
est surjectif (voir~\ref{ss.formel-spec});
il se factorise de manière unique 
par un morphisme fini et surjectif $\phi\colon \widetilde X\to\mathfrak X_s$
(la finitude se déduit du théorème~6.3.5/1 de~\cite{bosch-g-r1984},
voir aussi \cite{berkovich1994}, \S1, p.~541, ligne~12 en partant du bas).

Soit $U$ un ouvert affine de~$\mathfrak X_s$;
alors $\rho^{-1}(U)=\pi^{-1}(\phi^{-1}(U))$ 
et $\phi^{-1}(U)$ est un ouvert affine de~$\widetilde X$,
car $\phi$ est un morphisme affine,
puisque fini.
D'après le lemme ci-dessous, $\rho^{-1}(U)$ est un 
domaine affinoïde formel de~$X$.
\end{enumerate}
\end{exem}

\begin{lemm}
Soit $X$ un espace formellement analytique.
Pour que $X$ soit affinoïde, il faut et il suffit que sa réduction
soit un schéma affine.
\end{lemm}
\begin{proof}
C'est le théorème~3.1 de~\cite{bosch1977}.
\end{proof}

\subsection{}
Soit $X$ un espace formellement analytique et soit $V$ un domaine analytique de~$X$. Les conditions suivantes sont équivalentes:
\begin{enumerate}
\item Il existe un ouvert~$\Omega$ de~$\widetilde X$ tel que $V=\pi^{-1}(\Omega)$;
\item $V$ est G-recouvert par les domaines affinoïdes formels qu'il contient;
\item $\pi(V)$ est un ouvert de~$\tilde X$ et $V=\pi^{-1}(\pi(V))$.
\end{enumerate}
On dira que $V$ est un \emph{domaine analytique formel} de~$X$.

\begin{defi}
Soit $X$ et $Y$ des espaces formellement analytiques.
On dit qu'un morphisme d'espaces analytiques $f\colon X\to Y$ 
est \emph{formel} si les domaines affinoïdes formels~$U$
de~$X$ tels qu'il existe un domaine affinoïde formel~$V$ de~$Y$
contenant $f(U)$ constituent un G-recouvrement de~$X$.
\end{defi}

\begin{lemm}
Soit $X$ et $Y$ des espaces formellement analytiques
et soit $f\colon X\to Y$ un morphisme d'espaces analytiques.
Les conditions suivantes sont équivalentes:
\begin{enumerate}\def\theenumi{\roman{enumi}}
\item Pour tout domaine analytique formel~$V$
de~$Y$, l'image réciproque $f^{-1}(V)$ est un domaine analytique formel 
de~$X$;
\item Pour tout domaine affinoïde formel~$V$
de~$Y$, l'image réciproque $f^{-1}(V)$ est un domaine analytique formel 
de~$X$;
\item Il existe un G-recouvrement de~$Y$ par des domaines
affinoïdes formels~$V$ dont  l'image réciproque $f^{-1}(V)$ est un domaine analytique formel de~$X$;
\item $f$ est un morphisme formel.
\end{enumerate}
\end{lemm}
\begin{proof}
Les implications (i)$\Rightarrow$(ii)$\Rightarrow$(iii)
sont évidentes.

(iii)$\Rightarrow$(iv).
Soit $\mathscr V$ un G-recouvrement de~$Y$ par des domaines
affinoïdes formels de~$Y$ tel que $f^{-1}(V)$ soit un domaine 
analytique formel de~$X$, pour tout $V\in\mathscr V$.
Pour tout $V\in\mathscr V$, choisissons un G-recouvrement~$\mathscr U_V$
de~$f^{-1}(V)$ par des domaines affinoïdes formels de~$X$.
La réunion des~$\mathscr U_V$, pour $V\in\mathscr V$
est un G-recouvrement affinoïde formel de~$X$ 
qui garantit que $f$ est un morphisme formel.

(iv)$\Rightarrow$(ii).
Soit $\mathscr U$ un G-recouvrement de~$X$ par des domaines affinoïdes
formels tel que pour tout $U\in\mathscr U$, 
il existe un domaine affinoïde formel~$V_U$ de~$Y$ tel que $f(U)\subset V_U$.
Soit $W$ un domaine affinoïde formel de~$Y$; démontrons que son
image réciproque est un domaine analytique formel de~$X$.

Soit $U\in\mathscr U$; l'intersection
$ W \cap V_U $ est réunion finie $\bigcup_{i\in I_U} W_{i}$ 
de domaines analytiques de~$V_U$ de la forme $W_{i}=\{\abs {\phi_{i}}=1\}  $, 
où $\phi_i$ est une fonction analytique sur~$V_U$ 
de norme spectrale inférieure ou égale à~$1$.
L'image réciproque~$U_i=f^{-1}(W_i)\cap U$ de~$W_i$ dans~$U$ 
est donc définie par $\abs{f^*(\phi_i)}=1$ et
la norme spectrale de~$f^*(\phi_i)$ est inférieure ou égale à~$1$.
Comme $f^{-1}(W)\cap U = \bigcup_{i\in I_U} U_i$, cela démontre
que $f^{-1}(W)\cap U$ est un domaine analytique formel de~$U$.

Comme $f^{-1}(W)$ est G-recouvert par les $f^{-1}(W)\cap U$,
pour $U\in\mathscr U$, il en résulte que $f^{-1}(W)$
est un domaine analytique formel de~$X$, ce qu'il fallait démontrer.

(ii)$\Rightarrow$(i).
Soit $V$ un domaine analytique formel de~$Y$ et
soit $\mathscr W$ un G-recouvrement de~$V$ par des domaines affinoïdes
formels de~$Y$. Pour tout $W\in\mathscr W$, choisissons un G-recouvrement~$\mathscr U_W$ de $f^{-1}(W)$ par des domaines affinoïdes formels de~$X$.
Alors, la réunion de~$\mathscr U_W$ est un G-recouvrement
de~$f^{-1}(V)$ par des domaines affinoïdes formels de~$X$,
ce qui prouve que $f^{-1}(V)$ est un domaine analytique formel de~$X$.
\end{proof}

\subsection{}
La catégorie des espaces formellement analytiques possède des produits.
Si $X$ et $Y$ sont des espaces formellement analytiques,
leur produit a pour espace analytique sous-jacent le produit $X\times_k Y$,
et ses domaines formellement affinoïdes 
sont les produits  de domaines formellement affinoïdes de~$X$ et~$Y$ 
respectivement.

Par construction, on dispose d'un morphisme canonique 
$\widetilde {X\times_k Y}\to \widetilde X\times_{\tilde k}\widetilde Y$;
il est fini.

\subsection{}
Soit $A$ une $k$-algèbre strictement affinoïde.
Toute présentation
$f\colon k\{T_1,\dots,T_n\}\to A$ induit sur~$A$ une norme quotient
qui domine la semi-norme spectrale de~$A$;
on dit que $A$ est distinguée, 
ou que l'espace affinoïde $\mathscr M(A)$ est distingué,
s'il existe une telle présentation
pour laquelle la semi-norme spectrale de~$A$ est égale
à la norme quotient.
Si l'algèbre~$A$ est distinguée, elle est réduite,
sa semi-norme spectrale est une norme et prend ses valeurs dans~$\abs k$.
Inversement, si $k$ est un corps stable et $A$ est une algèbre
strictement affinoïde réduite dont la norme spectrale
prend ses valeurs dans~$\abs k$, alors $A$ est distinguée
(théorème~6.4.3/1 de~\cite{bosch-g-r1984}).

Si $A$ est distinguée, 
alors la $k^\circ$-algèbre~$A^\circ$ est topologiquement de présentation finie. Plus précisément, la démonstration du corollaire~2.6 de~\cite{bosch1977},
prouve que toute famille finie $(a_1,\dots,a_n)$
dans~$A^\circ$ dont les images dans~$\widetilde A$ engendrent~$\widetilde A$
engendre topologiquement~$A^\circ$ (comme $k^\circ$-algèbre)
et~$A$ (comme $k$-algèbre). De même, l'inclusion $k^{\circ\circ}A^\circ\to
A^{\circ\circ}$ est une égalité,
et l'homomorphisme canonique de~$\widetilde k\otimes_{k^\circ} A^\circ$
dans~$\widetilde A$ est un isomorphisme;
en particulier, $\widetilde k\otimes_{k^\circ}A^\circ$
est réduit.

Inversement, si $R$ est une $k^\circ$-algèbre, plate et topologiquement
de présentation finie, telle que $\widetilde k\otimes_{k^\circ} R$
soit réduite, alors, l'algèbre affinoïde~$R_k$ est distinguée
et l'homomorphisme canonique de~$R$ dans~$(R_k)^\circ$
est un isomorphisme. (Voir \cite{bosch-g-r1984}, proposition~6.4.3/4.)

\subsection{}
On dit qu'un espace formellement analytique~$X$
est distingué s'il possède un G-recouvrement formellement affinoïde
composé de domaines affinoïdes distingués.
Si $X$ est distingué, tout domaine formellement affinoïde de~$X$
est distingué (\cite{bosch1977}, corollaire~2.5 et commentaires 
pages~16--17).
Il s'ensuit que tout domaine analytique formel de~$X$ est 
un espace formellement analytique distingué.

\subsection{}
Soit $X$ un espace formellement analytique distingué.
Pour tout domaine formellement affinoïde~$U$ de~$X$,
la $k^\circ$-algèbre $\mathscr O(U)^\circ$ est topologiquement
de présentation finie et 
l'homomorphisme canonique $\widetilde k\otimes_{k^\circ}\mathscr O(U)^\circ
\to \widetilde{\mathscr O(U)}$
est un isomorphisme.
Les spectres formels $\Spf(\mathscr O(U)^\circ)$
se recollent en un schéma formel~$\mathfrak X$,
plat et localement topologiquement de présentation finie sur~$k^\circ$,
dont la fibre générique s'identifie à l'espace formellement analytique~$X$
et la fibre spéciale à~$\widetilde X$
de manière compatible aux applications de réduction.

Inversement, si $\mathfrak X$ est un $k^\circ$-schéma formel,
plat, localement topologiquement de présentation finie 
et à fibre spéciale réduite,
alors l'espace formellement analytique~$\mathfrak X_\eta$
est distingué et le schéma formel qui lui est associé s'identifie
à~$\mathfrak X$.

Cette construction induit une équivalence 
entre la catégorie des espaces formellement analytiques distingués sur~$k$
et celle des schémas formels plats, localement topologiquement de présentation
finie sur~$k^\circ$ dont la fibre spéciale est réduite.

\subsection{}
Soit $X$ un espace formellement analytique distingué 
dont la fibre spéciale est géométriquement réduite.
Alors, pour toute extension complète~$K$ de~$k$,
$X_K$ est distingué, et sa fibre  spéciale s'identifie à~$\widetilde X_{\tilde K}$ (\cite{Bosch-1976}, lemme~1.6).

Soit $X,Y$ des espaces formellement analytiques distingués
dont les  fibres spéciales sont géométriquement réduites
(il suffit en fait que l'une des deux fibres spéciales 
soit géométriquement réduite).
Alors, pour toute extension complète~$K$ de~$k$,
l'espace formellement analytique $X\times_k X$ 
est distingué et sa fibre spéciale s'identifie
à $\widetilde X\times_{\tilde k}\widetilde Y$
(loc. cit.).

\subsection{}
Un groupe formellement analytique est un objet en groupes dans la catégorie
des $k$-espaces formellement analytiques.
Son espace analytique sous-jacent est alors un groupe $k$-analytique.
Étant donné un groupe analytique~$G$, une structure de groupe
formellement analytique sur~$G$ est une structure d'espace formellement
analytique pour laquelle la loi de groupe de~$G$ est un morphisme
d'espaces formellement analytiques.

\begin{theo}[\cite{Bosch-1976}, Satz~1.1]\label{theo.bosch}
Tout morphisme d'espaces $k$-analytiques en groupes
entre groupes formellement analytiques est formellement analytique.
En particulier, un groupe $k$-analytique possède au plus une
structure de groupe formellement analytique.
\end{theo}

\subsection{}
Suivant~\cite{Bosch-1976}, définition~1.7,
on dit qu'un $k$-groupe formellement analytique~$G$ réduit a \emph{bonne réduction} s'il est distingué et si sa fibre spéciale est géométriquement réduite.

Dans ce cas, la réduction de l'espace formellement analytique $G\times_k G$
s'identifie au produit $\widetilde G\times_{\tilde k} \widetilde G$,
(\cite{Bosch-1976}, lemme~1.6).
La loi de groupe sur~$G$
induit alors sur~$\widetilde G$ une structure de $\tilde k$-schéma en groupes.
Comme il est géométriquement réduit, il est lisse.

Si $G$ a bonne réduction, alors $G_K$ a bonne réduction, pour toute
extension complète~$K$ de~$k$, de réduction  $\widetilde G_{\tilde K}$.

\def\alg{{\mathrm{a}}}

\section{Le théorème d'uniformisation}
\label{app-varab}

Les variétés abéloïdes sont les analogues 
pour la géométrie non archimédienne
des tores complexes.
Les analytifiés de variétés abéliennes 
en fournissent des exemples cruciaux qui seront ceux qui nous
intéresseront particulièrement.

Soit $k$ un corps ultramétrique complet dont la valeur absolue
est non triviale.

\begin{defi}
Une \emph{variété abéloïde} sur~$k$
est un $k$-groupe analytique connexe, propre et lisse.
\end{defi}

Une variété abéloïde est commutative (\cite{Lutkebohmert-2016}, corollaire~7.1.4).

L'analytifiée d'une $k$-variété abélienne est abéloïde.

La fibre générique d'un $k^\circ$-schéma formel en groupes 
propre, lisse et connexe est une variété abéloïde. Les
variétés abéloïdes obtenues de cette façon sont des groupes
formellement analytiques à bonne réduction; nous
les appellerons simplement variétés abéloïdes à bonne réduction. 

\subsection{}
Les variétés abéloïdes possèdent une uniformisation,
inspirée par celle que l'on connaît dans le cas complexe,
et qui a d'ailleurs été l'une des motivations
pour le développement d'une géométrie analytique non archimédienne.

Cette uniformisation a d'abord été construite
par Tate~\cite{tate95} pour les courbes elliptiques puis
étendue aux variétés abéloïdes par Raynaud~\cite{raynaud71}.
La théorie a été reprise et rédigée en détail par Bosch \&\ Lütkebohmert :
\cite{Lutkebohmert-2016} (chapitre~7)
et~\cite{berkovich1990} (chapitre~6);
voir aussi~\cite{Lutkebohmert-1995} dans le cas où
la valuation de~$k$ est discrète, 
ainsi que~\cite{BoschLutkebohmert-1984,BoschLutkebohmert-1986}
dans le cas des variétés abéliennes.

La description de cette uniformisation
fait jouer un rôle central aux tores et à leurs variantes
affinoïdes ; nous allons commencer par quelques rappels à leur sujet.

\subsection{}
On note $\gm[k,1]$ le sous-groupe analytique compact
de l'analytifié~$\gm[k]^\an=\mathscr M(k\{T,T^{-1}\})$ du tore algébrique~$\gm[k]=\Spec(k[T,T^{-1}])$  défini par la condition $\abs{T}=1$.
On appelle $k$-\emph{tore affinoïde} un $k$-groupe affinoïde~$T$
tel qu'il existe une extension complète~$K$ de~$k$
pour laquelle $T_K$ est isomorphe à une puissance~$(\gm[K,1])^n$;
on dit alors que $K$ déploie~$T$, ou que $T$ est déployé sur~$K$.
Si $T$ est un $k$-tore affinoïde, il existe en fait 
une extension finie et séparable de~$k$ qui déploie~$T$.

Dans~\cite{Bosch-1971}, Bosch définit les $k$-tores affinoïdes en
imposant que l'extension~$K$ comme ci-dessous soit le complété
d'une clôture algébrique de~$k$. Dans ce cas, le fait qu'on puisse
supposer~$K$ finie et séparable est le théorème~2.6 de cette référence.
Comme nous n'avons pas pu trouver une preuve du cas  général 
dans la littérature, explicitons-la ici.

Soit $k_1$ une extension finie séparable de~$k$
tel que $G_{k_1}$ soit un $k_1$-groupe analytique à bonne réduction.
Quitte à agrandir~$K$, on peut supposer qu'il contient 
une clôture algébrique de~$k_1$. 
La fibre spéciale $\widetilde{G_K}$ est un tore algébrique déployé
sur~$\tilde K$; comme elle est égale à $(\widetilde{G_{k_1}})_{\tilde K}$,  
cela démontre que $\widetilde{G_{k_1}}$ est un $\widetilde{k_1}$-tore algébrique
(non nécessairement déployé). Quitte à agrandir~$k_1$, on peut donc
supposer que $\widetilde{G_{k_1}}$ est un $\widetilde{k_1}$-tore algébrique
déployé. Nous allons démontrer que cette extension~$k_1$
convient alors. Quitte à remplacer~$k$ par~$k_1$, on peut
ainsi supposer que $G$ est un $k$-groupe affinoïde à bonne réduction
dont la fibre spéciale est un $\widetilde k$-tore algébrique déployé,
et il s'agit de démontrer que $G$ est un $k$-tore affinoïde déployé.

Soit $\mathfrak G$ le $k_0$-schéma formel associé au groupe à bonne
réduction~$G$ (c'est-à-dire le spectre formel de $\mathscr O(G)^\circ$).
Soit $u$ un isomorphisme de $\widetilde k$-groupes algébriques
de~$\gm[\widetilde k]^n$ sur~$\widetilde{G}$.
D'après~\cite[exposé~X, prop.~2.1]{sga3}, 
il existe pour tout élément $a\in k^{\circ\circ}$
un unique isomorphisme~$u_a$ de $k^\circ/(a)$-schémas en groupes
de $\gm[{k^\circ/(a)}]^n$ sur $\mathfrak G_{k^\circ/(a)}$ 
qui relève~$u$. La famille des~$u_a$ définit alors un isomorphisme
de schémas en groupes formels~$\mathfrak u$ de~$\mathfrak G_{\mathfrak m}^n$
sur~$\mathfrak G$. En passant aux fibres génériques, il induit
un isomorphisme de~$\gm[{k,1}]^n$ sur~$G$.

\subsection{}
Cet argument démontre en fait qu'un $k$-groupe  affinoïde
à bonne réduction~$G$ est un tore déployé si et seulement si
sa réduction est un tore déployé.

Plus généralement, si $G$ est un $k$-groupe affinoïde à bonne réduction,
c'est un $k$-tore si et seulement si sa réduction est un $\tilde k$-tore;
si c'est le cas, $G$ est déployé par une extension finie non ramifiée de~$k$.

\subsection{}
Soit $T$ un $k$-tore algébrique. 
Si $T$ est déployé et $j\colon T\to \gm^n$ est un isomorphisme de tores,
l'image réciproque de~$\gm[k,1]^n$ est un domaine
affinoïde de~$T^\an$ et est un $k$-tore affinoïde. Il ne dépend
pas du choix de~$j$; on le note~$T_1$.
Dans le cas général, si $k'$ est une extension finie galoisienne
qui déploie~$T$, le tore affinoïde $T_{k',1}$ descend à~$k$
et s'identifie à domaine affinoïde~$T_1$ de~$T^\an$ qui est un $k$-tore
affinoïde déployé par~$k'$.

Comme dans~\cite[Satz 3.8]{Bosch-1971},
cette construction établit une équivalence de catégories
de la catégorie des $k$-tores algébriques vers la catégorie des $k$-tores
affinoïdes.  En effet, ces deux catégories sont également identifiées,
par les foncteurs des caractères,  à la catégorie des $\Z$-modules
libres de type fini munis d'une action de Galois, de façon compatible
à la construction $T\mapsto T_1$.

\begin{lemm}\label{lemm.maximal-semiab}
Soit $A$ une variété abéloïde sur~$k$
et soit $N$ un domaine analytique compact de~$A$
qui est un sous-groupe analytique.
Supposons que $N$ soit un groupe formellement analytique
à réduction semi-abélienne.

\begin{enumerate}
\item
Pour tout groupe formellement analytique compact~$Z$
à bonne réduction et tout morphisme  de groupes $f\colon Z\to A$,
on a $f(Z)\subset N$.
\item
Pour toute extension complète~$k'$ de~$k$,
le domaine~$N_{k'}$ est le plus grand domaine
analytique compact de~$A_{k'}$ qui est un sous-groupe analytique
et un groupe formellement analytique à bonne réduction.
\end{enumerate}
\end{lemm}
\begin{proof}
L'assertion~\emph a) est la proposition~5.4.6 de~\cite{Lutkebohmert-2016},
appliquée au point rationnel~$z_0=e$. Comme $N_{k'}$
satisfait les mêmes propriétés que~$N$, l'assertion~\emph b)
se déduit de~\emph a) en prenant pour~$f$ un morphisme d'inclusion.
\end{proof}

\begin{prop}\label{prop.neron0}
Soit $A$ une variété abéloïde sur~$k$.
Il existe un unique domaine analytique compact~$N$ de~$A$
vérifiant les propriétés suivantes:
\begin{enumerate}
\item C'est un sous-groupe analytique de~$A$;
\item Il possède une structure de groupe formellement analytique;
\item Il existe une extension finie séparable~$k'$ de~$k$
telle que~$N_{k'}$ soit un groupe formellement analytique 
extension (dans la catégorie des groupes formellement analytiques)
d'une variété abéloïde à bonne réduction  par un tore affinoïde.
\end{enumerate}
La formation de~$N$ commute à toute extension des scalaires.
\end{prop}

Notons que dans~\emph c), $N_{k'}$ est un groupe formellement
analytique à bonne réduction semi-abélienne; il est en particulier
justiciable du lemme~\ref{lemm.maximal-semiab}.

On l'appellera le \emph{domaine de Néron}\index{domaine de Néron}
de la variété abéloïde~$A$.
Supposons que $k$ soit de valuation discrète et que $A$
soit l'analytifié d'une variété abélienne; soit $\mathscr A$
son modèle de Néron sur~$k^\circ$ et supposons que la
fibre spéciale de sa composante neutre~$\mathscr A^\circ$ 
soit une variété semi-abélienne. Dans ce cas, la fibre générique
du complété formel de~$\mathscr A^\circ$, qui est égale
à l'image réciproque de~$\mathscr A^\circ_s$ par l'application
de réduction, est le domaine de Néron de~$A$.

\begin{proof}
Le théorème~7.6.4 de~\cite{Lutkebohmert-2016}
affirme qu'il existe une extension finie séparable~$K$ de~$k$
et un domaine analytique compact~$M$ de~$B=A_{K}$ 
qui est un groupe formellement analytique, 
extension d'une variété abéloïde à bonne réduction par un tore affinoïde.
Cette propriété reste vraie si l'on agrandit~$K$.

La proposition s'en déduit par descente comme suit.

Démontrons d'abord l'assertion d'unicité. Soit $N$ un domaine
analytique vérifiant les conditions de la proposition.
Quitte à agrandir~$K$, on peut supposer que $k'\subset K$.

Comme $N_{k'}$ est distingué, $N_{K}$ est un domaine analytique
de~$A_{K}$, qui en est un sous-groupe analytique,
muni d'une structure de groupe formellement analytique distingué
dont la réduction est une variété semi-abélienne;
c'est donc~$M_{K}$ d'après le lemme~\ref{lemm.maximal-semiab},
appliqué à~$N_{K}$ et à~$M_{K}$.
On en déduit que $N$ est l'image de~$M$
par le morphisme de changement de base $X_K\to X$.

Pour établir l'existence d'un domaine analytique~$N$ vérifiant les conditions
de la proposition, nous allons démontrer que « $M$ descend à~$k$ ».
%
%
%
%
Quitte à agrandir~$K$, on peut supposer que l'extension $k\to K$ 
est galoisienne.
Soit $g\in\Gal(K/k)$ et soit $\gamma$ un automorphisme de corps valué
de~$K$ qui étend~$g$. Le domaine analytique $(gN)_K$ est égal
à $\gamma M$ et satisfait les propriétés de~$M$,
donc est égal à~$M$, d'après le lemme~\ref{lemm.maximal-semiab}.
On en déduit que $gN=N$.
Par une version « tordue » du théorème de pleine fidélité de Bosch
(théorème~\ref{theo.bosch}), le groupe~$\Gal(K/k)$ agit sur~$N$
par des morphismes formellement analytiques.
Le lemme~2.3 de~\cite{BoschLutkebohmert-1986} entraîne
alors qu'il existe un domaine analytique~$N$ de~$A$
qui est un sous-groupe analytique,
muni d'une structure formellement analytique,
tel que $M=N_{K}$ en tant que groupe formellement analytique.

Soit $N$ le domaine de Néron de~$A$ 
et soit $K$ une extension complète de~$k$. 
Comme $N_K$ vérifie les conditions de l'énoncé
qui caractérisent le domaine de Néron de la variété abéloïde~$A_K$,
c'est son domaine de Néron.
Cela conclut la démonstration.
\end{proof}

\subsection{}
Soit $A$ une variété abéloïde sur~$k$ et soit $N$
le domaine de Néron de~$A$ (proposition~\ref{prop.neron0}).
\emph{La catégorie des immersions fermées $j\colon T\to N$, où $T$ est un tore
affinoïde et $j$ un morphisme de groupes analytiques, admet
un objet final.}

Pour cette assertion, un argument de descente galoisienne
permet de supposer que $N$ a bonne réduction semi-abélienne déployée.
Compte tenu de~\cite{sga3}, exposé~XV, théorème~8.8,
la catégorie étudiée s'identifie alors,
via le foncteur de réduction, à celle
des immersions fermées de $\tilde k$-tores 
$\widetilde j\colon \widetilde T\to \widetilde N$.
Détaillons un peu l'argument.

L'anneau~$k^\circ$ n'est pas noethérien, 
mais le schéma formel~$\mathfrak N$ associé 
à la structure formelle analytique de~$N$
est topologiquement de présentation finie.
Par suite,  pour tout 
élément non nul~$a$ de~$k^\circ$, le $k^\circ/(a)$-schéma~$\mathfrak N_a$
déduit de~$\mathfrak N$ est de présentation finie;
il est lisse, car il est plat et que sa fibre spéciale est lisse.
Couplé à un argument d'approximation noethérienne,
le théorème~8.8 de l'exposé~XV de~\cite{sga3} 
assure que le foncteur de réduction $\mathfrak T\mapsto \mathfrak T_s$ 
induit une bijection
entre sous-tores formels de~$\mathfrak N$ et sous-tores de~$\widetilde N$.
Autrement dit, le foncteur $T\mapsto \widetilde T$
induit une bijection
entre sous-tores affinoïdes (Zariski-fermés)
à bonne réduction de~$N$ et sous-tores de~$\widetilde N$.

Comme le schéma en groupes semi-abélien~$\widetilde N$ possède un plus grand
sous-tore, le groupe formellement analytique~$N$ possède un plus grand
sous-tore affinoïde (Zariski-fermé) à bonne réduction.
Il reste à démontrer que tout sous-tore affinoïde affinoïde (Zariski-fermé)
de~$N$ est à bonne réduction. Soit $T$ un tel sous-tore affinoïde.
Soit $k'$ une extension finie galoisienne de~$k$ telle que $T_{k'}$
ait bonne réduction. 
Par ce qui précède, ce sous-tore de~$N_{k'}$ correspond à un sous-tore
de~$\widetilde {N_{k'}}=\widetilde N_{\widetilde k'}$, isomorphe à~$\widetilde{T_{k'}}$.
Comme $\widetilde N$ est déployé, il existe un sous-tore~$S$
de~$\widetilde N$ tel que $\widetilde{T_{k'}}=S_{\widetilde k'}$.
On en déduit que $T$ est l'unique sous-tore affinoïde (Zariski-fermé)
à bonne réduction de~$N$ dont la réduction est~$S$.

\subsection{}
Notons $j\colon T\to N$ l'immersion fermée maximale d'un
tore affinoïde  dans~$N$. 
\emph{Le quotient $B=N/T$ est une variété abéloïde sur~$k$ à potentielle
bonne réduction. Plus précisément, si $k'$ est une extension complète de~$k$
telle que $N_{k'}$ ait bonne réduction, alors $B_{k'}$ a bonne réduction.}

Détaillons l'argument. 

Supposons d'abord que $N$ soit extension
d'une variété abéloïde à bonne réduction~$B$ par un tore affinoïde~$S$;
alors $S=T$ et $N/T=B$.

Dans le cas général, il existe une extension finie galoisienne~$k'$
de~$k$ telle que le groupe analytique $(N/T)_{k'}$ soit 
une $k'$-variété abéloïde~$B'$ à bonne réduction.
Cette variété abéloïde est munie d'une donnée de descente 
galoisienne de~$k'$ à~$k$.
Comme la réduction de~$B'$ est projective, tout ensemble
fini de points de~$B'$ est contenu dans un domaine affinoïde
formel de~$B'$, image réciproque d'un ouvert affine
contenant les réductions des points considérés. 
Cela entraîne  que la donnée de descente sur~$B'$ est effective,
dans la catégorie des espaces formellement analytiques,
si bien qu'il existe une variété abéloïde~$B$ sur~$k$,
qui est formellement analytique,
et telle que $B'=B_{k'}$.

Notons que si $B'$ est l'analytifié d'une variété abélienne,
il en va de même de~$B$ par descente galoisienne schématique.

\subsection{}
Soit $T_\alg$ le tore algébrique tel que $T=T_{\alg, 1}$.
On déduit du théorème~7.6.4, (b) de~\cite{Lutkebohmert-2016}
et d'un argument de descente galoisienne finie
que tout morphisme de groupes analytiques
de~$T$ vers~$A$ s'étend de manière unique
en un morphisme de groupes analytiques de~$T_\alg^\an$ dans~$A$.
En particulier, l'immersion fermée de~$T$ dans~$N$
s'étend de manière unique 
en un morphisme de groupes analytiques~$q$ de~$T_\alg^\an$ dans~$A$.

Considérons le morphisme de groupes
analytiques $q'\colon T_\alg^\an \times N \to A$
déduit de~$q$  et de l'injection de~$N$ dans~$A$.
Le morphisme $t\mapsto (t^{-1},t)$ de $T$ dans~$T_\alg^\an\times N$
est une immersion fermée, son image est un sous-groupe analytique 
(Zariski-fermé) qu'on note~$D$. Le quotient~$E$ de~$T_\alg^\an\times N$
par~$D$ est le produit contracté $T_\alg^\an \times^{T} N$;
on a par construction une suite exacte de $k$-groupes analytiques
\[ 1 \to T_\alg^\an \to E \to B \to 0. \]
Le morphisme $q'\colon T_\alg^\an\times N\to A$ passe au quotient
et induit un morphisme analytique~$p$ de~$E$ dans~$A$.

\subsection{}
Dans ce paragraphe et dans les deux qui suivent, on suppose
que le domaine de Néron~$N$ a bonne réduction
et que le tore~$T$ est déployé. 

Soit $\mathfrak V$ un recouvrement formellement affinoïde
de~$B$ trivialisant l'extension~$N$;
pour tout $V\in \mathfrak V$, soit $f_V\colon V\to N$ une section
de la projection $q|_N\colon N\to B^\an$.
Sur $V\cap W$, les sections~$f_V$ et~$f_W$ diffèrent de la multiplication
par un morphisme $f_{VW}\colon V\cap W \to T$.
Ces morphismes $(f_{VW})$ forment un cocycle à valeurs dans~$T$
qui représente l'extension~$N$, mais également, 
si on les considère comme un cocycle à valeurs dans~$T_\alg^\an$,
l'extension~$E$.

Pour $V\in\mathfrak V$, la section~$f_V$ permet d'identifier $q^{-1}(V)$
au produit $V\times T_\alg^\an$; composons la projection sur le second facteur
avec l'application de tropicalisation de~$T$ vers~$T_{\alg,\trop}$;
on obtient une application~$\tau_V$ de $q^{-1}(V)$ dans~$T_{\alg,\trop}$.
Lorsqu'on identifie~$T_\alg^\an$ à~$\gm^r$, l'espace~$T_{\alg,\trop}$
s'identifie à~$\R^r$ et chaque coordonnée de~$\tau_V$ est 
le logarithme de la valeur absolue d'une fonction holomorphe inversible
(un caractère du tore).
Comme le cocycle $(f_{VW})$ est à valeurs dans~$T$,
ces applications $\tau_V$ se recollent en une application
$\tau\colon E\to T_{\alg,\trop}$.
(Dans le langage de ce travail, l'application~$\tau$ est~\pl.)

La restriction de~$\tau$ au tore~$T_\alg^\an$ 
est l'application de tropicalisation usuelle 
de~$T_\alg^\an$ sur~$T_{\alg,\trop}$. 
Comme $\tau|_{T_\alg^\an}$ est propre
et que $B$ est compacte, l'application $\tau\colon E\to T_{\alg,\trop}$
est surjective et propre.

Si $E=B$, $T_\alg$ est de dimension~$0$ et $\tau$ est l'application constante.

\subsection{}
Nous nous sommes placés dans le cas déployé, 
et le théorème~7.6.4, \emph c), de~\cite{Lutkebohmert-2016}
affirme alors que le morphisme $p\colon E\to A$ est surjectif 
et que son noyau~$\Lambda$ est un
sous-groupe discret, abélien libre de rang~$\dim(T)$, de~$E(k)$.
Le $\Lambda$-torseur~$E$ est un $\Z^{\dim(T)}$-torseur topologique. 
L'image du réseau~$\Lambda$ par~$\tau$
est un réseau cocompact de~$T_{\alg,\trop}$ 
(corollaire~7.6.2 de~\cite{Lutkebohmert-2016}).
Par suite, $\tau$ induit une application surjective
de~$A^\an$ dans le tore réel $T_{\alg,\trop}/\tau(\Lambda)$.

\subsection{}
Supposons encore que le domaine de Néron~$N$ ait bonne réduction
et que $T$ soit déployé.
L'étude analytique de l'extension~$E$
(\cite{berkovich1990}, \S6.5, p.~124) démontre
que $E$ est contractile; cela
repose sur le fait que la variété abéloïde~$B$ ayant bonne réduction,
l'espace analytique~$B$ est contractile, 
de même que l'analytifié d'un tore.
Cela fournit ainsi l'uniformisation analytique de~$A$:
$E$ est un revêtement universel de~$A$.

Plus précisément, Berkovich construit une rétraction par déformation
de~$B$ sur le point de Gauss~$\beta$
correspondant au point générique de sa réduction,
une rétraction par déformation de~$T_\alg^\an$ 
sur son squelette~$(T_\alg^\an)_\sq$,
et les combine pour obtenir une rétraction par déformation de~$E$ 
sur un sous-espace fermé~$S(E)$
qui est le squelette de $(T_\alg^\an)_\beta$.
La restriction de~$\tau$ à~$S(E)$
est un homéomorphisme  de ce squelette sur~$T_{\alg,\trop}$.

Cette rétraction par déformation est invariante par  l'action de~$\Lambda$
par translation et identifie~$\Lambda$ à un réseau~$\Gamma$ de~$S(E)$.
Elle passe au quotient et induit donc une rétraction par déformation
de~$A$ sur le tore réel $S(A)=S(E)/\Gamma$.

\subsection{}
Revenons maintenant à la situation générale.

Par descente galoisienne, il résulte du cas déployé que 
le morphisme $p\colon E\to A$ est surjectif,
que son noyau~$\Lambda$ est un $k$-groupe étale
et que $p$ fait de~$E$ un $\Lambda$-torseur.
Pour toute extension~$k'$ qui déploie la situation, $\Lambda_{k'}$
est un $k'$-groupe constant, abélien libre de rang~$\dim(T)$.

Considérons une extension galoisienne finie~$k'$
déployant la situation. Alors,  $E_{k'}$
est contractile, $p_{k'}\colon E_{k'}\to A_{k'}$ est un revêtement
universel de~$A_{k'}$.
Cependant, $p\colon E\to A$ n'est pas, en général, un isomorphisme
local (par exemple, dès que le noyau de~$p$ n'est pas déployé).

Cependant, après changement de base de~$k$ à une extension galoisienne~$k$
finie, les constructions du cas déployé fournissent
une rétraction par déformation sur le squelette $S(A_{k'})$
qui est un tore réel de dimension~$\dim(T)$.
Berkovich démontre que les constructions données commutent
à l'action de $\Gal(k'/k)$ et fournissent une rétraction par déformation
de~$A$ sur le quotient du tore réel~$S(A_{k'})$
par l'action de~$\Gal(k'/k)$.

\section{Fibrés en droites sur une variété abéloïde}

\subsection{}
Nous allons maintenant rappeler comment l'uniformisation
d'une variété abéloïde permet de décrire la catégorie des fibrés en droites
sur icelle, d'après~\cite{Raynaud-1971c,bosch-l91,gubler2010,Lutkebohmert-2016}.
Nous nous plaçons dans le cas déployé, c'est-à-dire que le tore $T$
est déployé et que la variété abéloïde~$B$ a bonne réduction.
Partons du diagramme fondamental que fournit l'uniformisation:
\[ \begin{tikzcd} 
& & \Lambda \ar{d}{j} \ar{dr}{\phi} \\
    1 \ar{r} & T_\alg^\an \ar{r} & E \ar{r} \ar{r}{q} & B \ar{r} & 0 
\end{tikzcd}
\]

On dispose également d'un diagramme dual, classique
dans la théorie des 1-motifs:
\[ \begin{tikzcd} 
& & \Lambda' \ar{d}{j'} \ar{dr}{\phi'} \\
  1 \ar{r} & T_\alg^{\prime\an} \ar{r} & E' \ar{r}  \ar{r}{q'} & B' \ar{r} & 0 
\end{tikzcd}
\]
Rappelons sa construction.

 Notons $B'$ la variété
abéloïde duale de~$B$ et $\mathscr P$ la biextension de Poincaré
sur $B\times B'$; 
elles sont canoniquement munies de structures formellement analytiques
que nous désignons par les mêmes notations.
Nous avons noté $j\colon \Lambda\to E$ l'immersion canonique
et posé $\phi=q\circ j\colon \Lambda\to B$.

Notons $\Lambda'$ le groupe des caractères du tore~$T_\alg$
et $T'_\alg$ le tore algébrique de groupe de caractères~$\Lambda$.
À tout $\lambda'\in\Lambda'$ correspond par fonctorialité
une extension~$\lambda'_*E$ 
de~$B$ par~$\gm$, donc un fibré en droites $\phi'(\lambda')
\in B'(k)$. L'application $\phi'\colon \Lambda' \to B'(k)$
est un morphisme de groupes.
De plus, le fibré en droites $q^*\phi'(\lambda')$ possède
une trivialisation canonique~$s_{\lambda'}$ et
l'application $\lambda' \mapsto s_{\lambda'}$ est additive.

Par bidualité, $B(k)$ s'interprète comme $\Pic^0(B')$
de sorte que le morphisme de groupes $\phi\colon\Lambda\to B(k)$
fournit une extension~$E'$ de~$B'$ par le tore~$T_\alg^{\prime\an}$.
Pour $(\lambda,\lambda')\in\Lambda\times\Lambda'$,
le point $j(\lambda)$ de~$E(k)$ fournit un 
$k$-point de l'extension~$\lambda'_*E$ de~$B$ par~$\gm$,
c'est-à-dire un point $\langle\lambda,\lambda'\rangle$
de la fibre de~$\mathscr P$
en $(\phi(\lambda),\phi'(\lambda'))$ et,
par symétrie, un $k$-point~$j'(\lambda')$ de l'extension~$E'$
qui relève $\phi'(\lambda')$. L'application $j'\colon\Lambda'\to E'$
ainsi construite est un morphisme de groupes, injectif.

On note enfin l'égalité $\langle \lambda,\lambda'\rangle=s_{\lambda'}(j(\lambda))$.

\subsection{}
Soit $L$  un fibré en droites sur~$A$; il vérifie le théorème
du cube (\cite{Lutkebohmert-2016}, théorème~7.1.6, \emph a)).
Comme tout fibré en droites cubiste sur~$E$,
le fibré en droites~$p^*L$ sur~$E$ provient de~$B$
et même d'un fibré en droites cubiste~$M$ 
sur l'espace formellement analytique~$B$
(\cite{Lutkebohmert-2016}, proposition~6.2.10).
Par la formule $b\mapsto t_b^*M\otimes M^{-1}$, 
le fibré en droites~$M$ fournit un morphisme de variétés
abéloïdes $\theta_L\colon B\to B'$,
$t_b$ désignant la translation par~$b$.
Le fibré~$M$ n'est pas unique; les ambiguïtés possibles
consistent exactement à le multiplier par
un fibré en droites de la forme~$\phi'(\lambda')$, pour $\lambda'\in\Lambda'$,
si bien que l'application~$\theta_L$ ne dépend bien que de~$L$.

D'autre part, ce fibré en droites~$p^*L$  est muni d'une donnée de
descente relative à l'action de~$\Lambda$, qu'on transporte
à $q^*M$ au moyen de l'isomorphisme donné.
Ces données de descente sont explicitées par la proposition~6.2.12 de~\cite{Lutkebohmert-2016}:
elles sont paramétrées par des couples $(u,r)$,
où $u\colon\Lambda\to\Lambda'$ est un morphisme de groupes
tel que 
\[ \theta_M\circ\phi=\phi'\circ u \]
et $r\colon\Lambda\to\phi^*M$ est une trivialisation
telle que
\[ r(\lambda+\mu)\otimes r(\lambda)^{-1}\otimes r(\mu)^{-1}\otimes r(0)
 = \langle \lambda, u(\mu)\rangle. \]
Pour donner sens à cette relation,
observons que le membre de droite de l'égalité précédente
est une section de $\mathscr P$ au-dessus
du point $(\phi(\lambda),\phi'\circ u(\mu))$.
Le membre de gauche 
est une trivialisation du fibré en droites
\begin{multline*}
 \phi(\lambda+\mu)^*M\otimes \phi(\lambda)^*M^{-1}\otimes \phi(\mu)^*M^{-1} \otimes \phi(0)^*M \\
\simeq \phi(\lambda)^* (t_{\phi(\mu)}^* M \otimes M^{-1})
	\otimes 0^* (t_{\phi(\mu)}^*M\otimes M^{-1})^{-1} \\
\simeq \phi(\lambda)^* \theta_M(\phi(\mu))  \otimes \theta_M(\phi(\mu))^{-1}.\end{multline*}
Le premier facteur est la fibre en~$\phi(\lambda)$
du fibré en droites $\theta_M\circ\phi(\mu)
=\phi'(u(\mu))$, c'est-à-dire la fibre de~$\mathscr P$
en $(\phi (\lambda),\phi'(u(\mu)))$;
le second facteur est la fibre en l'origine 
du fibré en droites~$\theta_M\circ\phi(\mu)$,
c'est-à-dire la fibre de~$\mathscr P$ en~$(0,\phi'(u(\mu)))$;
par la structure de biextension sur~$\mathscr P$, le membre de gauche
s'interprète  donc comme un point de la fibre de~$\mathscr P$
en $(\phi(\lambda),\phi'\circ u(\mu))$.

À un tel couple $(u,r)$ est associé la famille des isomorphismes
\[ c_\lambda \colon q^*M\to t_{j(\lambda)}^* q^*M,
\quad
m\mapsto r(\lambda)\otimes r(0)^{-1} \otimes s_{u(\lambda)} \otimes m, \]
pour $\lambda\in\Lambda$.

\subsection{}
Soit $\lambda\in\Lambda$ et $\lambda'\in\Lambda'$;
on a construit une section canonique $\langle \lambda,\lambda'\rangle$
du fibré de Poincaré~$\mathscr P$ 
au-dessus de~$(\phi(\lambda),\phi'(\lambda'))$.
Grâce à la structure formelle du fibré de Poincaré, il lui est
associé une métrique formelle. 
L'application 
\[ (\lambda,\lambda') \mapsto \log (\norm{\langle \lambda,\lambda'\rangle}) \]
est une forme bilinéaire sur~$\Lambda\times\Lambda'$.
Comme $\Lambda$ est un réseau de~$E$, cette forme bilinéaire
non dégénérée (\cite{Lutkebohmert-2016}, prop.~6.1.8,
équivalence de (i) et~(ii)).

Pour $(u,r)$ comme ci-dessus, l'application
\[  (\lambda,\mu) \mapsto \log (\norm{\langle \lambda, u(\mu)\rangle} )
\]
est une forme bilinéaire symétrique sur~$\Lambda\times\Lambda$
(\cf \cite{Lutkebohmert-2016},  prop.~6.4.1, \emph c)).
Nous noterons~$q_L$ la forme quadratique 
sur~$T_{\alg,\trop}$ correspondante:
\[ q_L (\tau(\lambda))=  \log (\norm{\langle \lambda, u(\lambda)\rangle} ).\]

\begin{theo}
Soit $L$ un fibré en droites sur~$A$, soit $(M,u,r)$ le triplet sur~$E$
fournissant~$L$ par descente.
Pour que $L$ soit ample, il faut et il suffit que $M_s$ soit ample sur
la variété abélienne résiduelle~$B_s$
et que la forme quadratique~$q_L$ sur~$T_{\alg,\trop}$
soit définie positive. Si c'est le cas, la variété
abéloïde~$B$ et le schéma formel abélien~$\mathfrak B$ sont algébrisables.
\end{theo}
Pour la démonstration, nous renvoyons à~\cite{Lutkebohmert-2016}, 
théorème~6.4.4. Dans cette référence, l'amplitude de~$M_s$
intervient par celle du fibré en droites analytique~$M$ 
sur la variété abéloïde~$B$; d'après le lemme suivant, ces
deux assertions sont équivalentes et entraîne l'énoncé d'algébricité.

\begin{lemm}
Soit $\mathfrak B$ un schéma formel abélien sur~$k^\circ$
et soit $\mathfrak L$ un fibré en droites sur~$\mathfrak B$.
Les conditions suivantes sont équivalentes:
\begin{enumerate}\def\theenumi{\roman{enumi}}\def\labelenumi{(\theenumi)}
\item Le fibré en droites~$\mathfrak L_\eta$ 
sur la variété abéloïde~$\mathfrak B_\eta$ est ample;
\item Le fibré en droites~$\mathfrak L$ est relativement ample
sur $\Spf(k^\circ)$;
\item Le fibré en droites~$\mathfrak L_s$ sur la variété
abélienne~$\mathfrak B_s$ est ample.
\end{enumerate}
\end{lemm}
\begin{proof}
L'implication (iii)$\Rightarrow$(ii)  résulte du
théorème des fonctions formelles de Grothendieck
sous sa forme non nécesssairement noethérienne
(\cite{abbes2010}, \S2.11).
L'implication (ii)$\Rightarrow$(i) est évidente.
Il reste à démontrer que si $\mathfrak L_\eta$ est ample,
alors $\mathfrak L_s$ est ample.
Dans les deux théories, le critère d'amplitude est analogue:
il faut et il suffit qu'une puissance du fibré soit effective
et qu'il définisse un morphisme injectif dans la variété
abéloïde (resp.\ abélienne) duale.
Supposons que $\mathfrak L_\eta^{\otimes n}$ est effectif.
Comme $\mathfrak B$ est propre et lisse, à fibres géométriquement
connexes, l'« adhérence schématique » d'un diviseur de Cartier qui le définit
est un diviseur de Cartier~$\mathfrak D$ sur~$\mathfrak X$
et $\mathscr O_{\mathfrak X}(\mathfrak D)$ est l'unique
fibré en droites sur~$\mathfrak X$ qui prolonge~$\mathfrak L_\eta^{\otimes n}$
(\cite{Lutkebohmert-1990}, preuve du lemme~2.9).
On a donc $\mathfrak L^{\otimes n}\simeq \mathscr O_{\mathfrak X}(\mathfrak D)$,
ce qui entraîne que $\mathfrak L_s^{\otimes n}$ est effectif.
Considérons le morphisme $\theta_{\mathfrak L}$ 
du schéma formel abélien~$\mathfrak B$ vers le schéma formel abélien
dual~$\mathfrak B'$. Sur la fibre générique, on 
obtient le morphisme correspondant entre variétés abéloïdes;
son noyau est fini par hypothèse, 
si bien que $\theta_{\mathfrak L,\eta}$ est surjectif.
Par surjectivité de la spécialisation, $\theta_{\mathfrak L,s}$
est également surjectif, donc son noyau est fini.
Cela prouve que $\mathfrak L_s$ est ample.
\end{proof}

\begin{rema}
Lorsque $A$ est une variété abélienne, nous allons
démontrer que la variété abéloïde~$B$ et son extension~$E$
sont algébrisables.

Lorsque la situation est déployée, on vient de voir
que la variété abéloïde~$B$ est  simultanément
l'analytifié de la fibre générique d'un schéma abélien~$\mathscr B$
et la fibre générique de son complété formel $\widehat{\mathscr B}$.
Les descriptions d'une extension d'une
variété abélienne par un tore déployé sont les mêmes
dans les cas algébrique, formel et analytique:
ces extensions correspondent 
à un morphisme du groupe des caractères de ce tore
dans la variété duale.
Cela fournit une extension~$\mathscr E$ du schéma abélien~$\mathscr B$
par un $k^\circ$-tore  déployé~$\mathscr T$.
La fibre générique de cette extension est une extension algébrique
de~$B$ par le tore~$T_\alg$ dont l'analytifiée est~$E$;
la fibre générique du complété formel de cette extension
est une extension de~$B$ par le tore affinoïde~$T$
et s'identifie comme extension au domaine de Néron~$N$ de~$A$.

Dans le cas général, il existe donc une extension finie galoisienne~$k'$
de~$k$ telle que l'extension~$E_{k'}$ de $B_{k'}$ par
le tore~$T_{\alg,k'}^\an$ soit algébrisable et que ce tore soit déployé.
Ces espaces et l'extension sont munis d'une donnée de descente
galoisienne analytique. Comme $B_{k'}$ est projective,
cette donnée de descente est algébrique, et donc effective,
si bien que $B$ est algébrisable.

Les automorphismes d'une extension d'une variété abélienne par~$\gm$
(induisant l'identité sur les deux facteurs) sont en particulier
des automorphismes de $\gm$-torseurs. Par \textsc{gaga} appliqué
au fibré en droite correspondant, ces automorphismes d'extension
analytiques sont algébrisables. Il en va de même des automorphismes
d'extension par un tore analytique déployé, de sorte que
la donnée de descente de l'extension~$E_{k'}$ est algébrique.
Comme $E_{k'}$ est affine sur la variété projective~$B_{k'}$,
cette donnée de descente est encore effective. Autrement dit,
l'extension~$E$ est algébrique.

Cependant, il n'y a pas de raison pour que le domaine de Néron
soit la fibre générique d'un schéma formel de type fini,
a fortiori, algébrisable.

Bien sûr, le morphisme d'uniformisation $p\colon E\to A$ n'est
pas algébrique si le tore est de dimension strictement positive,
puisque son noyau est infini et de dimension~$0$.
\end{rema}

\begin{prop}\label{prop.unif-pic0}
Soit $L$ un fibré en droites sur~$A$, soit $(M,u,r)$
le triplet sur~$E$ fournissant~$L$ par descente. 
Pour que $L$ soit invariant par translation,
il faut et il suffit que $\mathfrak M$ soit invariant par translation
et que la forme quadratique~$q_L$ soit nulle.
\end{prop}
\begin{proof}
Cela résulte du corollaire~6.2.13 de~\cite{Lutkebohmert-2016}
qui affirme l'équivalence de trois assertions (a), (b), (c);
l'assertion~(a) signifie que $L$ est invariant par translations
et l'assertion~(b) signifie que $M$ est invariant par translation
et~$q_L$ est nulle.
\end{proof}

\clearpage
\bibliographystyle{smfplain}
\bibliography{aclab,acl,chern}

\clearpage
\printindex

\end{document}